\pdfoutput=1
\documentclass[10pt,reqno]{amsart}

\numberwithin{equation}{section}

\usepackage[margin=1in]{geometry}
\usepackage{cancel}
\usepackage[tt=false]{libertine}
\usepackage{mathtools}
\usepackage{kbordermatrix}
\usepackage{amssymb}
\usepackage{mathrsfs}
\usepackage[varbb]{newpxmath}
\usepackage{textcomp}
\usepackage{enumerate}
\usepackage{bm}
\usepackage[usenames,dvipsnames]{xcolor}
\usepackage[colorlinks=true,
	linkcolor=NavyBlue,citecolor=PineGreen,
	urlcolor=RedViolet]{hyperref}
\hypersetup{bookmarksopen=true}
\usepackage[footnotesize,
	justification=centering]{caption}
\usepackage[font=footnotesize]{subcaption}
\usepackage{graphicx}

\usepackage{tikz}
\tikzset{c/.style={every coordinate/.try}}
\usetikzlibrary{shapes.geometric}
\usetikzlibrary{calc}
\usetikzlibrary{shapes,arrows}

\tikzstyle{block} = [rectangle, draw,
	fill=white, text width=14em, text centered,
	minimum height=4em]
\tikzstyle{line} = [draw, -latex']

\newtheorem{thm}{Theorem}[section]
\newtheorem{lem}[thm]{Lemma}
\newtheorem{cor}[thm]{Corollary}
\newtheorem{ppn}[thm]{Proposition}

\newtheorem{maintheorem}{Theorem}
\theoremstyle{definition}
\newtheorem{dfn}[thm]{Definition}
\newtheorem{rmk}[thm]{Remark}
\newtheorem*{rmk*}{Remark}

\DeclareMathOperator*{\argmax}{\textup{arg\,max}}
\DeclareMathOperator{\supp}{\textup{supp}}
\DeclareMathOperator{\RelInt}{\textup{ri}}
\DeclareMathOperator{\Th}{\textup{th}}
\DeclareMathOperator{\girth}{\textup{girth}}
\DeclareMathOperator{\length}{\textup{len}}
\DeclareMathOperator{\rk}{\textup{rk}}
\DeclareMathOperator{\PROJ}{\textup{\textsf{proj}}}
\DeclareMathOperator{\COARSEN}{\textup{\textsf{coarsen}}}
\DeclareMathOperator{\CUBE}{\textup{\textsf{cube}}}
\DeclareMathOperator{\CO}{\textup{\textsf{co}}}
\DeclareMathOperator{\diam}{\textup{diam}}
\DeclareMathOperator{\Var}{Var}
\DeclareMathOperator{\Cov}{Cov}
\DeclareMathOperator{\ch}{ch}
\DeclareMathOperator{\proc}{\textup{\textsf{pr}}}
\DeclareMathOperator{\switch}{\textup{\textsf{sw}}}
\DeclareMathOperator{\rad}{\textup{rad}}
\DeclareMathOperator{\disc}{\textup{\textsf{disc}}}
\DeclareMathOperator{\corr}{\textup{\textsf{corr}}}

\newcommand{\bemph}[1]{\textbf{\textup{#1}}}
\newcommand{\DS}{\displaystyle}
\newcommand{\beq}{\begin{equation}}
\newcommand{\eeq}{\end{equation}}
\newcommand{\f}{\frac}
\renewcommand{\emptyset}{\varnothing}
\renewcommand{\vec}[1]{\underline{\smash{#1}}}
\newcommand{\set}[1]{\{#1\}}
\newcommand{\Ind}[1]{\mathbf{1}\{#1\}}
\newcommand{\pd}{\partial}
\renewcommand{\P}{\mathbb{P}}
\newcommand{\E}{\mathbb{E}}
\newcommand{\BARC}{\hspace{-3pt}\begin{array}{c}}
\newcommand{\EARC}{\end{array}\hspace{-3pt}}
\newcommand{\st}{\textup{\textsf{t}}} 
\renewcommand{\log}{\ln}
\newcommand{\Ent}{\mathcal{H}} 
\newcommand{\DKL}{\mathcal{D}_\textup{KL}} 

\newcommand{\albd}{\alpha_\textup{lbd}}
\newcommand{\aubd}{\alpha_\textup{ubd}}
\newcommand{\arsb}{\alpha_\star}
\newcommand{\adrsb}{\alpha_\textup{clust}}
\newcommand{\acond}{\alpha_\textup{cond}}
\newcommand{\asat}{\alpha_\textup{sat}}

\newcommand{\BX}{x}
\newcommand{\uBX}{\vec{\BX}}
\newcommand{\LABEL}{\bm{l}}
\newcommand{\qav}[1]{\langle#1\rangle} 

\newcommand{\gm}{\gamma}
\newcommand{\Gm}{\Gamma}
\newcommand{\ep}{\epsilon}
\newcommand{\lm}{\lambda}
\newcommand{\Lm}{\Lambda}
\newcommand{\bh}{\nu}
\newcommand{\bLambda}{\bm{\Lm}}

\newcommand{\SPIN}[1]{\textup{\texttt{\footnotesize #1}}}
\newcommand{\SSPIN}[1]{\textup{\texttt{\footnotesize #1}}}
\newcommand{\plus}{\SPIN{+}}
\newcommand{\minus}{\SPIN{-}}
\newcommand{\PM}{\SPIN{\textpm}}
\newcommand{\free}{\SPIN{f}}
\newcommand{\red}{\SSPIN{r}}
\newcommand{\yel}{\SSPIN{y}}
\newcommand{\blu}{\SSPIN{b}}
\newcommand{\grn}{\SSPIN{g}}
\newcommand{\cya}{\SSPIN{c}}
\newcommand{\whi}{\SSPIN{w}}
\newcommand{\pur}{\SSPIN{p}}
\newcommand{\notr}{\SSPIN{u}}

\newcommand{\Yel}{\SPIN{Y}}
\newcommand{\Red}{\SPIN{R}}
\newcommand{\Whi}{\SPIN{W}}
\newcommand{\Cya}{\SPIN{C}}
\newcommand{\msg}{\SPIN{w}}
\newcommand{\dmp}{\dot{\msg}}
\newcommand{\hmp}{\hat{\msg}}
\newcommand{\RYGB}{\red,\yel,\grn,\blu}
\newcommand{\RYC}{\red,\yel,\cya}
\newcommand{\lit}{\SPIN{L}}
\newcommand{\mred}{\SPIN{v}}

\newcommand{\ZZ}{\bm{Z}}
\newcommand{\sepZZ}{\bm{Z}_\textup{sep}}
\newcommand{\extZZ}{\bm{Z}_\textup{ext}}
\newcommand{\ZNAE}{Z_\textsc{nae}}
\newcommand{\bde}{\bm{\delta}}

\newcommand{\dq}{\dot{q}}
\renewcommand{\dh}{\dot{h}}

\newcommand{\dbde}{\bm{\dot{\delta}}}
\newcommand{\ddbde}{\bm{\ddot{\delta}}}

\newcommand{\dbh}{\dot{\bh}}

\newcommand{\hq}{\hat{q}}
\newcommand{\hh}{\hat{h}}

\newcommand{\hcP}{\smash{\hat{\mathcal{P}}}}
\newcommand{\bhmu}{\bm{\hat{\mu}}}
\newcommand{\hPGW}{\widehat{\PGW}}
\newcommand{\hbh}{\hat{\bh}}

\newcommand{\usi}{\smash{\vec{\sigma}}}
\newcommand{\uL}{\smash{\vec{\bm{L}}}}
\newcommand{\uX}{\smash{\vec{X}}}
\newcommand{\uta}{\smash{\vec{\tau}}}
\newcommand{\ueta}{\smash{\vec{\eta}}}
\newcommand{\ud}{\smash{\vec{d}}}
\newcommand{\ux}{\smash{\vec{x}}}
\newcommand{\uy}{\smash{\vec{y}}}

\newcommand{\tq}{\tilde{q}}
\newcommand{\tW}{W}
\newcommand{\rwhq}{H}
\newcommand{\dbz}{\bm{\dot{z}}}
\newcommand{\hbz}{\bm{\hat{z}}}
\newcommand{\bbz}{\bm{\bar{z}}}

\newcommand{\onersb}{\textup{1-\textsc{rsb}}}
\newcommand{\trsb}{\textup{$t$-\textsc{rsb}}}
\newcommand{\fullrsb}{\textup{full-\textsc{rsb}}}
\newcommand{\infrsb}{\textup{$\infty$-\textsc{rsb}}}
\newcommand{\ksat}{\textup{$k$-\textsc{sat}}}
\newcommand{\knae}{\textup{$k$-\textsc{nae-sat}}}
\newcommand{\nae}{\textup{\textsc{nae-sat}}}

\newcommand{\err}{\textup{\textsf{err}}}
\newcommand{\ev}{\textup{\textsf{eval}}}

\newcommand{\PGW}{{\textup{\textsf{\footnotesize PGW}}}}
\newcommand{\uPGW}{{\textup{\textbf{PGW}}}}
\newcommand{\SOL}{\textup{\textsf{\footnotesize SOL}}}
\newcommand{\NAE}{\textup{\textsf{\footnotesize NAE}}}
\newcommand{\CLUSTERS}{\textup{\textsf{\footnotesize CL}}}
\newcommand{\LIN}{\textup{\textsf{LIN}}} 
\newcommand{\ff}{\textup{\textsf{f}}}
\newcommand{\embed}{\textup{\textsf{emb}}}
\newcommand{\ConfigModel}{\textup{\textbf{\textsf{CM}}}}
\newcommand{\MAT}{\textup{\textsf{M}}} 
\newcommand{\AUGMENT}{\textup{\textsf{aug}}} 

\newcommand{\hWP}{\widehat{\textup{\footnotesize\textsf{WP}}}}
\newcommand{\dWP}{\dot{\textup{\footnotesize\textsf{WP}}}}
\newcommand{\BP}{\textup{\footnotesize\textsf{BP}}}

\newcommand{\Qstar}{{}_\star\hspace{-1pt}Q}
\newcommand{\dqstar}{{}_\star\hspace{-1pt}\dq}
\newcommand{\qstar}{{}_\star\hspace{-1pt}q}
\newcommand{\hqstar}{{}_\star\hspace{-1pt}\hq}

\newcommand{\starpi}{{}_\star\hspace{-1pt}\pi}
\newcommand{\omstar}{{}_\star\hspace{-1pt}\omega}
\newcommand{\bhstar}{{}_\star\hspace{-1pt}\bh}
\newcommand{\dbhstar}{{}_\star\hspace{-1pt}\dbh}
\newcommand{\hbhstar}{{}_\star\hspace{-1pt}\hbh}
\newcommand{\Gmstar}{{}_\star\hspace{-1pt}\Gm}

\newcommand{\psistar}{{}_\star\hspace{-1pt}\psi}
\newcommand{\Lmstar}{{}_\star\hspace{-1pt}\Lm}
\newcommand{\Testar}{{}_\star\hspace{-1pt}\Theta}
\newcommand{\lmstar}{{}_\star\hspace{-1pt}\lm}
\newcommand{\etastar}{{}_\star\hspace{-1pt}\bmeta}
\newcommand{\stareta}{{}_\star\hspace{-1pt}\eta}
\newcommand{\bhustar}{{}_\star\hspace{-1pt}\bhu}
\newcommand{\nustar}{{}_\star\hspace{-1pt}\nu}
\newcommand{\starzeta}{{}_\star\hspace{-1pt}\zeta}
\newcommand{\bstar}{{}_\star b}
\newcommand{\xstar}{{}_\star x}

\newcommand{\prodom}{{}_{*}\omega}
\newcommand{\prody}{{}_{*}y}
\newcommand{\prodnu}{{}_{*}\nu}

\newcommand{\prodmu}{{}_{*}\mu}
\newcommand{\prodhq}{{}_{*}\hq}
\newcommand{\prodhQ}{{}_{*}\hat{Q}}
\newcommand{\prodLm}{{}_{*}\Lambda}
\newcommand{\prodTe}{{}_{*}\Theta}
\newcommand{\proddq}{{}_{*}\dq}
\newcommand{\prodq}{{}_{*}q}
\newcommand{\prodQ}{{}_{*}Q}

\newcommand{\qbul}{{}_\bullet\hspace{-1pt}q}
\newcommand{\dqbul}{{}_\bullet\hspace{-1pt}\dq}
\newcommand{\hqbul}{{}_\bullet\hspace{-1pt}\hq}
\newcommand{\hqstarstar}{{}_{\star\star}\hspace{-1pt}\hq}

\newcommand{\TSQ}{{}_{\scalebox{0.4}{$\square$}}\hspace{-1pt}}
\newcommand{\SQpi}{\TSQ\pi}

\newcommand{\SQdq}{\TSQ \dq}
\newcommand{\SQhq}{\TSQ \hq}
\newcommand{\SQeta}{\TSQ \bmeta}
\newcommand{\SQbhu}{\TSQ \bm{\hat{u}}}
\newcommand{\SQT}{\TSQ T}

\newcommand{\bL}{\bm{L}} 
\newcommand{\bT}{\bm{T}} 
\newcommand{\bt}{\bm{t}} 

\newcommand{\poisP}{\mathbb{P}}
\newcommand{\wP}{\mathbb{Q}}

\newcommand{\bsp}{\textup{\textsf{BSP}}}
\newcommand{\ACT}{\textup{\textsf{Act}}}

\newcommand{\vrelerr}{\dddot{\textup{\textsf{\footnotesize ERR}}}}
\newcommand{\crelerr}{\widehat{\textup{\textsf{\footnotesize ERR}}}}
\newcommand{\CRELERR}{\widehat{\textup{\textbf{\textsf{ERR}}}}}
\newcommand{\VRELERR}{\dddot{\textup{\textbf{\textsf{ERR}}}}}
\newcommand{\mdel}{{}^\textup{m}\hspace{-2pt}\delta}
\newcommand{\mdelred}{{}^\textup{m}\hspace{-2pt}\dot{\delta}}
\newcommand{\mbde}{{}^\textup{m}\hspace{-2pt}\bm{\delta}}
\newcommand{\mbdered}{{}^\textup{m}\hspace{-2pt}\bm{\dot{\delta}}}
\newcommand{\adel}{\chi}

\newcommand{\cdel}{\kappa}
\newcommand{\cdelred}{\dot{\cdel}}
\newcommand{\fdel}{\rho}
\newcommand{\avhq}{\bar{q}}
\newcommand{\avhz}{\bar{z}}
\newcommand{\optnu}{\nu^\textup{op}}
\newcommand{\optdbh}{\dbh^\textup{op}}
\newcommand{\opthbh}{\hbh^\textup{op}}
\newcommand{\vrt}{v_\textup{\textsc{rt}}}
\newcommand{\crt}{a_\textup{\textsc{rt}}}
\newcommand{\ert}{e_\textup{\textsc{rt}}}
\newcommand{\vth}{\vartheta_*}

\newcommand{\cluster}{\mathscr{C}}
\newcommand{\DD}{\mathscr{D}}

\newcommand{\GG}{\mathscr{G}}
\newcommand{\HH}{\mathscr{H}}
\newcommand{\RR}{\mathscr{R}}
\newcommand{\LL}{\mathscr{L}}

\newcommand{\albet}{\mathcal{X}}
\newcommand{\cD}{\mathcal{D}}
\newcommand{\cX}{\mathcal{X}}
\newcommand{\cY}{\mathcal{Y}}
\newcommand{\cZ}{\mathcal{Z}}

\newcommand{\cP}{\mathcal{P}}

\newcommand{\II}{\mathfrak{i}}
\newcommand{\CC}{\mathfrak{c}}
\newcommand{\BTW}{\mathfrak{b}}

\newcommand{\FF}{\mathbf{F}}

\newcommand{\rprime}{R'}
\newcommand{\NotOrd}{\cancel{O}} 
\newcommand{\NotSC}{\cancel{S}} 

\newcommand{\REC}{\mathcal{R}}
\newcommand{\hREC}{\smash{\hat{\mathcal{R}}}}
\newcommand{\dREC}{\smash{\dot{\mathcal{R}}}}
\newcommand{\Rec}{\bm{R}}

\newcommand{\bmu}{\bm{\mu}}
\newcommand{\bmeta}{\bm{\eta}}
\newcommand{\bhu}{\bm{\hat{u}}}
\newcommand{\Pois}{\textup{\textsf{Pois}}}
\newcommand{\POIS}{\textup{\textsf{pois}}}
\newcommand{\POpm}{\textup{\textsf{po}}_{\PM}}
\newcommand{\clausedeg}{\textup{\textsf{cl}}}
\newcommand{\PINT}{\mathscr{P}}

\newcommand{\tprime}{T'}
\newcommand{\tree}{\mathscr{T}}
\newcommand{\Tout}{\tree_\textup{out}}
\newcommand{\Nout}{\mathscr{N}_\textup{out}}
\newcommand{\Yout}{Y_\textup{out}}
\newcommand{\Zout}{Z_\textup{out}}
\newcommand{\Tin}{\tree_\textup{in}}
\newcommand{\UU}{\mathscr{U}}
\newcommand{\NN}{\mathscr{N}}
\newcommand{\tL}{\tree(\bL)}
\newcommand{\cGG}{\GG'}
\newcommand{\GGR}{\GG''}
\newcommand{\DDtyp}{\DD_\star}
\newcommand{\Leaves}{\mathcal{L}}
\newcommand{\RN}{\mathfrak{X}} 
\newcommand{\ptree}{\mathscr{Q}}
\newcommand{\TIME}{\mathfrak{t}} 
\newcommand{\cM}{\mathcal{M}} 

\newcommand{\COHER}{\textup{\textsf{coher}}}
\newcommand{\PROB}{\textup{\textsf{\footnotesize PR}}}
\newcommand{\MSR}{\textup{\textsf{\footnotesize MSR}}}
\newcommand{\MARG}{\textup{\textsf{\footnotesize MARG}}}
\newcommand{\CCOLS}{\textup{\textsf{\footnotesize COLS}}}
\newcommand{\OPT}{\textup{\textsf{\footnotesize OPT}}}
\newcommand{\unifpi}{\pi^\textup{unif}}
\newcommand{\Simplex}{\bm{\Delta}}
\newcommand{\Judicious}{\bm{J}}
\newcommand{\BPq}{q_\textup{\textsc{bp}}} 

\newcommand{\BIG}{\textup{\textsf{\footnotesize BIG}}}
\newcommand{\DEG}{\textup{\textsf{\footnotesize DEG}}}

\newcommand{\blocks}{\mathfrak{B}} 
\newcommand{\ablock}{\mathcal{A}}
\newcommand{\MARK}{\textup{\textsf{mark}}}

\newcommand{\emax}{e_\star}

\newcommand{\DEFECTIVE}{\textup{\textsf{DEF}}}
\newcommand{\DELTACONST}{\delta_*}
\newcommand{\EPSCONST}{\ep_*}
\newcommand{\EPS}{\epsilon}
\newcommand{\EPSONE}{\epsilon_1}
\newcommand{\EPSTWO}{\epsilon_2}
\newcommand{\EPSTHREE}{\epsilon_3}
\newcommand{\EPSLIGHT}{\epsilon_\circ}
\newcommand{\EPSDIV}{\epsilon_4}
\newcommand{\EPSINT}{\epsilon}
\newcommand{\EPSP}{\epsilon_\infty}

\newcommand{\bP}{\mathbf{P}}
\newcommand{\hit}{\hat{\iota}}
\newcommand{\bzeta}{\bm{\zeta}}
\newcommand{\vsi}{\varsigma}
\newcommand{\uvsi}{\smash{\vec{\varsigma}}}
\newcommand{\barpi}{\bar{\pi}}
\newcommand{\col}{\Sigma}
\newcommand{\ucol}{\smash{\vec{\Sigma}}}
\newcommand{\ups}{\upsilon}
\newcommand{\vups}{\smash{\vec{\upsilon}}}
\newcommand{\bom}{\bm{\varpi}}
\newcommand{\notDiverse}{\cancel{\mathbb{D}}}
\newcommand{\notLight}{\cancel{\mathbb{L}}}
\newcommand{\COLS}{\mathscr{A}} 

\newcommand{\bPa}{\bP^a}
\newcommand{\cXa}{\cX^a}
\newcommand{\bPb}{\bP^b}
\newcommand{\cXb}{\cX^b}
\newcommand{\bPc}{\bP^c}
\newcommand{\cXc}{\cX^c}
\newcommand{\bPd}{\bP^d}
\newcommand{\cXd}{\cX^d}
\newcommand{\bPe}{\bP^f}
\newcommand{\cXe}{\cX^f}

\title{Proof of the satisfiability conjecture for large $k$}
\author[J. Ding]{Jian Ding$^*$}
\author[A. Sly]{Allan Sly$^\dagger$}
\author[N. Sun]{Nike Sun}

\date{\today}

\thanks{Research supported in part by $^*$NSF DMS-1313596; $^\dagger$DMS-1208338, DMS-1352013 and Sloan Fellowship}

\newcommand{\KAPPA}{\kappa_*}
\newcommand{\ZETA}{\zeta_*}

\begin{document}

\maketitle

\vspace{-23pt}
\begin{center}
\textit{\footnotesize University of Chicago;}\vspace{-3pt}\\
\textit{\footnotesize University of California--Berkeley and Australian National University;}\vspace{-3pt}\\
\textit{\footnotesize Microsoft Research and Massachusetts Institute of Technology}
\end{center}

\begin{abstract} We establish the satisfiability threshold for random $\ksat$ for all $k\ge k_0$, with $k_0$ an absolute constant. That is, there exists a limiting density $\asat(k)$ such that a random $\ksat$ formula of clause density~$\alpha$ is with high probability satisfiable for $\alpha<\asat$, and unsatisfiable for $\alpha>\asat$. We show that the threshold $\asat(k)$ is given explicitly by the one-step replica symmetry breaking prediction from statistical physics. The proof develops a new analytic method for moment calculations on random graphs, mapping a high-dimensional optimization problem to a more tractable problem of analyzing tree recursions. We believe that our method may apply to a range of random \textsc{csp}s in the $\onersb$ universality class. \end{abstract}

\section{Introduction}\label{s:intro} A \bemph{constraint satisfaction problem} (\textsc{csp}) consists of variables $\BX_1,\ldots,\BX_n$ subject to constraints $a_1,\ldots,a_m$. This general framework encompasses several fundamental problems in computer science, the most classic example being \bemph{boolean satisfiability} (\textsc{sat}). Other examples include various natural problems in graph combinatorics; such as (proper) coloring, independent set, and cut or bisection problems. In each of these \textsc{csp}s, the variables $\BX_i$ take values in some fixed alphabet $\albet$, and it is of interest to understand properties of the subset $\SOL\subseteq\albet^n$ of valid assignments: its total size, say, or the maximum value of some objective function.

In many cases, even deciding if $\SOL$ is nonempty --- which requires, \textit{a~priori}, exhaustive search over $\albet^n$ --- is \textsc{np}-complete~\cite{MR0378476}, and thus believed to require super-polynomial time in worst-case instances. This worst-case intractability of \textsc{csp}s was one of the early motivations to develop some ``average-case'' theory for \textsc{csp}s. For instance, one approach \cite{MR822205} is to study the typical runtime of algorithms in a \bemph{random}~\textsc{csp}: formally, a sequence $(\P^n)_{n\ge1}$ where each $\P^n$ is a probability measure on \textsc{csp}s of $n$ variables; the interest is in asymptotic behaviors as $n\to\infty$.

Since their introduction into the computer science literature, random \textsc{csp}s have become a subject of interest among physicists and mathematicians as well, inspired in part by early numerical experiments \cite{cheeseman1991proceedings,mitchell1992hard} suggesting phase transition phenomena. On the basis of heuristic analytic methods, physicists predict that they exhibit a rich array of phenomena. However, the rigorous analysis of random~\textsc{csp}s poses substantial mathematical difficulties, and many of the physics predictions remain challenging open problems. This paper considers one of these predictions which has been especially well-studied, the satisfiability threshold conjecture.

\subsection{Main result} The \bemph{random $\ksat$ model} is as follows: variables $\BX_1,\ldots,\BX_n$ take values $\textsc{true}\equiv\plus$ or $\textsc{false}\equiv\minus$. They are subject to constraints $a_1,\ldots,a_M$ where $M$ is a $\Pois(n\alpha)$ random variable. Conditioned on $M$, the constraints are independent. Each constraint is a random disjunctive clause: it is the boolean \textsc{or} of $k$ independent literals, with each literal sampled uniformly at random from $\set{\plus \BX_1,\minus \BX_1,\ldots,\plus \BX_n,\minus \BX_n}$. The clause is satisfied if at least one of its $k$ literals evaluates to $\plus$. The entire instance is satisfied if every clause is satisfied. This defines a probability measure $\poisP^{n,\alpha}$ over $\ksat$ problem instances;\footnote{Although $\poisP^{n,\alpha}$ depends on $k$ as well as $\alpha$, one typically considers the problem for fixed $k$, so we suppress this dependence from the notation.} and the sequence $(\poisP^{n,\alpha})_{n\ge1}$ is what is commonly termed the \bemph{random $\ksat$ model at density~$\alpha$}.

A $\ksat$ problem instance is naturally encoded by a bipartite graph $\GG=(V,F,E)$ where $V$ is the set of variables, $F$ is the set of clauses, and $E$ is the set of edges. The presence of an edge $e=(av)\in E$ indicates that variable $v$ participates in clause $a$. The edge always comes with a sign $\lit_e$ which is $\plus$ or $\minus$ depending on whether $\plus \BX_v$ or $\minus \BX_v$ appears in clause $a$. Thus $\poisP^{n,\alpha}$ can be regarded as the law of a random bipartite graph with signed edges. This is a bipartite analogue of the standard Erd\H{o}s--R\'enyi random graph; and for this reason we sometimes refer to this model as ``random Erd\H{o}s--R\'enyi $\ksat$.''

It has been notoriously challenging to characterize a most basic property of this model: what fraction of randomly sampled instances are satisfiable? Based on numerical simulations and non-rigorous arguments, it is proposed that for each fixed $k\ge2$, there is a critical value $\asat$ --- depending on $k$ but not on $n$ --- such that for all $\ep>0$,
	\[\lim_{n\to\infty}
	\P^{n,\asat-\ep}(\textup{satisfiable})
	=1=\lim_{n\to\infty}
	\P^{n,\asat+\ep}(\textup{unsatisfiable})\,.\]
In words, the model has a sharp transition from satisfiable to unsatisfiable, with high probability.\footnote{An event occurs \bemph{with high probability} if its probability tends to one in the limit $n\to\infty$.} This is known as the \bemph{satisfiability threshold conjecture}. For $k=2$, it is known to be true with threshold $\asat=1$ \cite{2677892sat,MR1255142}. It has been a long-standing open problem to establish a satisfiability threshold for any $k\ge3$. Our main result resolves this conjecture for large $k$:

\begin{maintheorem}[\bemph{main theorem}]\label{t:main} For $k\ge k_0$, random $\ksat$ has a sharp satisfiability threshold $\asat$, with explicit characterization $\asat=\arsb$ given by Proposition~\ref{p:phi} below.\end{maintheorem}

The study of the random $\ksat$ model has seen important contributions by researchers from several different communities --- probability theory, combinatorics, computer science, and statistical physics. In particular, the explicit characterization $\asat=\arsb$ emerged from the physics literature \cite{MPZ_Science,MR2213115}, via the so-called ``one-step replica symmetry breaking'' ($\onersb$) framework. Subsequent works \cite{MR2317690, montanari2008clusters} detailed the implications of $\onersb$ for the geometry of the solution space $\SOL$, and the resultant obstacles to locating the threshold. At the same time, a quite separate challenge posed by this model concerns the fluctuating local geometry of the underlying random (bipartite Erd\H{o}s--R\'enyi) $\ksat$ graph. This issue has been most notably considered within the probability and computer science communities \cite{1182003, MR2121043, MR2961553, MR3436404}.

The current paper is heavily guided by insights from the aforementioned works. We describe these connections in the remainder of this introductory section, which is organized as follows. In~\S\ref{ss:intro.rigorous} we survey the prior rigorous literature on random $\ksat$. We then turn to the statistical physics work on this problem (\S\ref{ss:intro.physics}), and describe the general notion of replica symmetry breaking (\textsc{rsb}) (\S\ref{ss:intro.rsb}). In the specific context of random $\ksat$, we explain (\S\ref{ss:intro.onersb}) how this manifests as \bemph{one-step} \textsc{rsb}, leading to an explicit threshold prediction (\S\ref{ss:intro.onersb.threshold}). Lastly we explain how the underlying graph geometry poses further challenges, and outline our proof strategy to deal with these issues (\S\ref{ss:intro.proof.overview}). 

\subsection{Prior rigorous results}\label{ss:intro.rigorous} Exact satisfiability thresholds have been rigorously shown in only a few models, including $k$-\textsc{xor-sat}~\cite{MR1972120, MR3455676} and random 1-in-$\ksat$~\cite{ACIM:01}. Also, as we remarked above, it has been proven for random \mbox{2-\textsc{sat}} \cite{2677892sat,MR1255142}, along with even finer results characterizing the scaling window \cite{MR1824274}. Compared with all these, however, random $\ksat$ for $k\ge3$ is believed to undergo a very different type of transition, as we explain below (\S\ref{ss:intro.rsb}).

For random $\ksat$, even the existence of $\asat$ was not known for any $k\ge3$. To date, the strongest result that applies for \emph{every} $k\ge2$ is Friedgut's theorem \cite{Friedgut:99}. It states that for each fixed $k\ge2$, there is a sharp threshold \emph{sequence} $\asat(n)$
such that for all $\ep>0$,
	\beq\label{e:friedgut}
	\lim_{n\to\infty}\P^{n,\asat(n)-\ep}(\textup{satisfiable})
	=1=\lim_{n\to\infty}
	\P^{n,\asat(n)+\ep}(\textup{unsatisfiable}).
	\eeq
The theorem does not imply that $\asat(n)$ converges to a unique limit, as the conjecture requires. It also gives no quantitative information on $\asat(n)$.

Complementing Friedgut's theorem, there have been many results giving quantitative bounds on $\asat(n)$, usually in the limit of large $k$. An easy calculation of the first moment of assignments gives a fairly accurate upper bound~\cite{FraPul:83}. Truncating the first moment to ``locally maximal'' solutions gives an even more precise bound 
	\[\limsup_{n\to\infty}\asat(n)
	\le 2^k\log 2-\f12(1+ \log 2) + \ep_k
	\quad\textup{\cite{KKKY:98}}\,.\]
In the above and throughout what follows, $\ep_k$ denotes any error term that tends to zero as $k\to\infty$. This upper bound is already correct in the second-order term. In contrast, all early \emph{lower} bounds for the $\ksat$ threshold, which were generally algorithmic in nature, missed the true threshold by a large multiplicative factor --- the current best algorithmic result \cite{MR2544853} gives a lower bound of order $2^k(\log k)/k$ while the threshold is of order $2^k$.

More recent advances in lower bounding $\asat$ have all taken a non-algorithmic route --- via the second moment method, in combination with Friedgut's theorem. This route, initiated by \cite{1182003}, faces two major challenges in the random $\ksat$ model. In brief, the first concerns the geometry of the solution space $\SOL\subseteq\set{\plus,\minus}^n$, while the second concerns the geometry of the underlying bipartite graph --- we explain these further below. Important advances on the second issue led to a series of improvements in the lower bound:
	\beq\label{e:intro.lbd.rs}
	\liminf_{n\to\infty}\asat(n)\ge
	\left\{\begin{array}{ll}
	2^{k-1}\log2 - O(1)
	&\textup{\cite{1182003}}\,;\\
	2^k\log 2-O(k)
	&\textup{\cite{MR2121043}}\,;\\
	2^k \log 2 -\tfrac32 \log 2 + \ep_k	&\textup{\cite{MR2961553}}\,.
	\end{array}\right.
	\eeq
These works did not address the first issue (the solution space geometry), which was first discussed in the physics literature. Coja-Oghlan and Panagiotou were the first to address both issues simultaneously: they prove
	\beq\label{e:intro.best.prev}
	\liminf_{n\to\infty}\asat(n)\ge
	2^k\log 2-\f12(1+ \log 2) - \ep_k
	\quad\textup{\cite{MR3436404}}\,,\eeq
matching the upper bound of \cite{KKKY:98} up to $\ep_k$. This gives the best estimate of the $\ksat$ threshold prior to the current work, which closes the $\ep_k$ gap for large $k$.

To explain the difficulties in pinning down an exact threshold, we turn next to a survey of the statistical physics heuristics for this model, leading to the explicit characterization of~$\asat$. Having done this, we can then give a more detailed account of the earlier advances in rigorous lower bounds, as well as the obstacles that remain. With this context, we give an overview of our proof approach at the conclusion of this section.

\subsection{Statistical physics}\label{ss:intro.physics} Statistical physicists became interested in random~\textsc{csp}s as examples of \bemph{spin glasses}, which are models of disordered systems (see e.g.\ \cite{mezard1985replicas,MPV1987}). Perhaps the most extensively studied such model is the \bemph{Sherrington--Kirkpatrick} (SK) spin glass \cite{sk1975}: let $(g_{ij})_{i,j\ge1}$ be an array of i.i.d.\ gaussian random variables with mean zero and variance $2/n$. The SK spin glass is defined as the probability measure on $\uBX\in\set{\plus,\minus}^n$ given by
	\[\mu(\uBX) = \f{1}{Z}
	\prod_{1\le i<j\le n}
	\exp\{ \beta g_{ij}
	\BX_i\BX_j\}\,,\]
where $Z$ is the partition function (normalizing constant). The measure $\mu$ of course depends on the $g_{ij}$, so it is a random measure supported on $\set{\plus,\minus}^n$. Parisi conjectured in a series of seminal papers \cite{parisi1979infinite, parisi1980order, parisi1980sequence, parisi1983order} that the SK measure has intricate asymptotics, characterized by an infinitely nested hierarchy. Key aspects of this prediction have been rigorously proved in celebrated works \cite{MR1957729, MR2195134, MR2999044}.

The analogue of the SK measure in random $\ksat$ is the uniform measure $\nu$ over the solution space $\SOL$, which can be regarded as a random measure on $\set{\plus,\minus}^n$:
	\beq\label{e:unif.sol}\nu(\uBX) =\f1Z
		\prod_{1\le j\le M}
		\mathbf{1}\bigg\{\hspace{-4pt}
		\begin{array}{c}
		\uBX\text{ satisfies}\\
		\text{clause $a_j$}
		\end{array}
		\hspace{-4pt}\bigg\}
	=\f{\Ind{\uBX\in\SOL}}Z 
	\eeq
where in this context $Z=|\SOL|$. In contrast with SK, however, it turns out that $\ksat$ exhibits the most interesting behavior when the number of constraints scales proportionally to the number of variables --- in other words, when the graph of interactions $\GG$ is \bemph{sparse}. This is a central distinction from the SK model, where the graph of interactions is the complete graph on $n$ vertices.

An extensive statistical physics literature demonstrates how heuristics for the SK model (and a larger family of \bemph{$p$-spin models}) can be adapted to the analysis of sparse random \textsc{csp}s such as random $\ksat$. In one sense, the sparsity of interactions makes these models more challenging to analyze. In the SK model, because each vertex has a large number of neighbors, there is a self-averaging effect which is crucial to the analysis. The effect does not occur on sparse graphs, and this turns out to pose major difficulties in the mathematical study of random $\ksat$ and other sparse models. For example, sparse versions of the SK and $p$-spin models have been studied \cite{MR3098679, MR3180967, MR3334277}, but remain not nearly as well understood as the complete graph versions (a very incomplete list of references includes, e.g., \cite{MR3052333,MR3320318,MR3555353,MR4078711,MR3487243,subag2021following,montanari2021optimization,MR4207445}).

In spite of this, random $\ksat$ is expected to exhibit behaviors which are very similar to those of SK, and in certain aspects significantly simpler. In particular, while the SK model is described by an infinitely nested hierarchy ($\infrsb$ or $\fullrsb$), many random \textsc{csp}s --- $\ksat$ included --- are conjecture to be described by a depth-one hierarchy ($\onersb$). This is central to our understanding of this problem, and we describe this next.

\subsection{Replica symmetry and cavity methods}
\label{ss:intro.rsb} Let $\nu$ be a random measure on $\set{\plus,\minus}^n$
(such as in \eqref{e:unif.sol}). For any function $f:(\set{\plus,\minus}^n)^\ell\to\mathbb{R}$, we let $\qav{f}_\nu$ denote the expected value of $f$ if its $\ell$ arguments are independent samples of $\nu$: explicitly,
	\[\qav{f}_\nu
	\equiv
	\sum_{\uBX^1,\ldots,\uBX^\ell}
	f(\uBX^1,\ldots,\uBX^\ell)
	\prod_{j\le\ell} \nu(\uBX^j)\,.\]
In the physics terminology, the $\uBX^j$ are \emph{replicas} of system $\nu$. Let $R$ be the \emph{overlap} (normalized inner product) between two replicas, $R=n^{-1}(\uBX^1\cdot\uBX^2)$. 
The measure $\nu$ is termed \bemph{replica symmetric} (\textsc{rs}) if this overlap is well-concentrated, in the sense of
	\beq\label{e:rs.overlap}
	\bm{v}_n=\E\bigg\{ \qav{R^2}_\nu
	-(\qav{R}_\nu)^2\bigg\}=o_n(1)\,.
	\eeq
Otherwise $\nu$ is said to be
\bemph{replica symmetry breaking} (\textsc{rsb}).
Note that one can rewrite $\bm{v}_n$ as the average, over all pairs $i,j\in[n]$, of the expected correlation between $x^1_i x^2_i$ and $x^1_j x^2_j$,
	\[\corr_{i,j} \equiv \E
	\bigg\{\qav{x^1_i x^1_j}_\nu
	\qav{x^2_ix^2_j}_\nu
	-\qav{x^1_i}_\nu\qav{x^1_j}_\nu
	\qav{x^2_i}_\nu\qav{x^2_j}_\nu
	\bigg\}\,.\]
The \textsc{rs} condition \eqref{e:rs.overlap}
says that the average correlation $\corr_{i,j}$ is small. Non-concentration of the overlap (\textsc{rsb}) indicates
the presence of \bemph{long-range correlations}.

If $[n]\equiv\set{1,\ldots,n}$ is the vertex set of a \bemph{sparse} graph, then most of the contribution to \eqref{e:rs.overlap} comes from vertices $i,j$ which are far apart in the graph. Thus, in the sparse setting, \textsc{rs} is regarded by physicists as being equivalent to \bemph{correlation decay}: if $\uBX$ is a sample from $\nu$, then $\BX_i,\BX_j$ are roughly independent if $i,j$ are far apart in the graph. In other words, \emph{in an \textsc{rs} model, the behavior around a vertex $i\in[n]$ depends only on its local neighborhood.} The commonly studied sparse random graph models are locally tree-like --- for example, the random $\ksat$ graph converges locally in law to a certain (multi-type) Galton--Watson tree. It is expected that sparse models in the \textsc{rs} regime can be accurately analyzed by a certain set of tree approximations, which generally go under the name of \bemph{belief propagation} (\textsc{bp}), or \bemph{replica symmetric cavity methods}. In this viewpoint, roughly speaking, the stochastic process on the finite graph is approximated by a stochastic process on the limiting tree.

We shall not go into many more details on the \textsc{rs} cavity method, pointing instead to the literature (\cite[Ch.~14]{MR2518205} and refs.\ therein) for details. We only note here that a key step in the method is to compare graphs $\GG,\GG'$ where $\GG'$ is $\GG$ with a random clause $a\in F$ removed. Let $\pd a$ denote the variables incident to $a$ in graph $\GG$: these variables are most likely well-separated in $\GG'$. Thus, in the \textsc{rs} (correlation decay) regime, one can treat these variables as \emph{(approximately) independent, with laws depending only on their local neighborhoods in $\GG'$.} This is a key simplification, leading to explicit tree recursions which can be analyzed. In \textsc{rs} models this is a powerful analytic tool. It leads further to an explicit prediction for the free energy of the model, expressed in terms of a fixed point of the tree recursions --- the \bemph{\textsc{rs} free energy} or \bemph{Bethe free energy} (see \cite{MR2246363}). In \textsc{rsb} models, however, it is expected that this method yields false predictions, as the lack of correlation decay invalidates the independence assumption. We discuss this next in the context of \textsc{sat}.

\subsection{Condensation and one-step replica symmetry breaking}\label{ss:intro.onersb} In a broad class of models, it is believed that \textsc{rsb} arises due to the formation of \bemph{clusters}, which are loosely defined as dense regions of the measure. Specifically, for $t\ge1$ integer, \bemph{$t$-step replica symmetry breaking} ($\trsb$) is the special case of \textsc{rsb} in which the overlap $R$ concentrates on exactly $t+1$ values. The way this can occur is that in the cube $\set{\plus,\minus}^n$, there is a hierarchy of scales $\delta_0\gg \delta_1\gg\ldots\gg \delta_t$ such that there are many clusters of mass~$\delta_i$ nested within each cluster of mass~$\delta_{i-1}$. The maximal scale $\delta_0$ refers to the entire space $\set{\plus,\minus}^n$. This scenario is often summarized by a depth-$t$ tree, where the vertices at depth $j$ correspond to the clusters at scale $\delta_j$. It was proposed by Parisi (refs.\ cited above) that the $t\to\infty$ limit describes symmetry breaking in the SK model.

By contrast, later works (\cite{MPZ_Science,MR2317690} and refs.\ therein) indicated that random $\ksat$ and several other sparse \textsc{csp}s of interest exhibit $\onersb$, corresponding to a single-depth hierarchy $\delta_0\gg\delta_1$. In fact, random $\ksat$ is believed to have a rich phase diagram (Figure~\ref{f:phase}) which includes both \textsc{rs} and $\onersb$ regimes. As we next describe, physicists predict a \bemph{condensation threshold} $\acond \in (0,\asat)$, which marks the onset of $\onersb$.

\begin{figure}[h!]\includegraphics[height=1.4in]{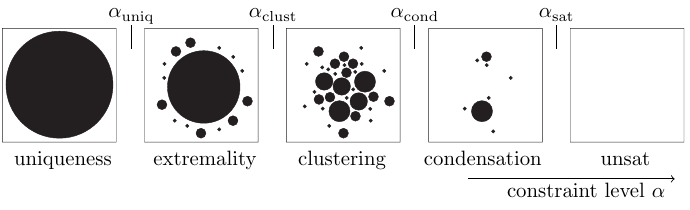}\caption{\textit{Figure adapted from \cite{MR2317690}.} Conjectural phase diagram of random $\ksat$: each panel depicts the typical geometry of the solution space $\SOL\subseteq\set{\plus,\minus}^n$ in a different regime of the constraint level $\alpha$. The \bemph{satisfiability threshold} $\asat$ is the point beyond which $\SOL$ is empty with high probability. The \bemph{condensation threshold} marks the onset of symmetry breaking: the model is \textsc{rs} for $\alpha<\acond$, and is $\onersb$ for $\acond<\alpha<\asat$.}\label{f:phase}\end{figure}

This conjectural phase diagram was derived \cite{zdeborova2007phase,montanari2008clusters,MR2317690} in the following manner. It starts from the hypothesis that the model is at most $\onersb$. This means that there is at most one hierarchy of clustering, or equivalently that \emph{clusters are replica symmetric}. This means that one can successfully apply \textsc{rs} inference methods, but at the level of clusters rather than individual solutions. Let $\bm{\Omega}$ count the total number of clusters in $\SOL$. For $0\le s\le \log2$, let $\bm{\Omega}_s$ count only those of size approximately $\exp\{ns\}$. By the \textsc{rs} cavity method applied at the level of clusters, it is possible to calculate an explicit function $\Sigma(s)$ such that $\bm{\Omega}_s$ concentrates around $\exp\{n\Sigma(s)\}$. Note the implicit dependence on $\alpha$; we write also $\Sigma(s)\equiv\Sigma(s;\alpha)$.

This calculation of $\Sigma(s)$, combined with other insights from the literature, lead physicists to suggest the phase diagram shown in Figure~\ref{f:phase} \cite{MR2317690}. As soon as $\alpha$ crosses a \bemph{clustering threshold} $\adrsb$, the curve $\Sigma(s)$ becomes positive for some interval of $s$-values. There is some rigorous evidence for the clustering and related phenomena \cite{4691011,Gamarnik:2014:LLA:2554797.2554831}. Note that if $\Sigma(s)\ge0$, the clusters of size $\exp\{ns\}$ contribute roughly $\exp\{n[s+\Sigma(s)]\}$ to the total number of solutions. If $\Sigma(s)<0$, clusters of size $\exp\{ns\}$ typically do not occur. One characterization of the condensation threshold is
	\[\acond
	=\inf\set{\alpha : s_1\ne s_\star }\]
where $s_1$ and $s_\star$ are both functions of $\alpha$, defined by
	{\setlength{\jot}{0pt}\begin{align*}
	s_1 &= \argmax_s
		\set{ s +\Sigma(s) : 0\le s\le \log2}\,,\\
	s_\star	&= \argmax_s
		\set{ s +\Sigma(s) :
		\textup{$0\le s\le \log2$ and
		$\Sigma(s)\ge0$}
		}\,.
	\end{align*}}%
For $\alpha<\acond$, the solution space $\SOL$ is dominated by clusters of size $\exp\{ns_1\}$, of which there are exponentially many ($\exp\{n\Sigma(s_1)\}$). Since each cluster carries a negligible fraction of the total mass, the asymptotic ($n\to\infty$) measure is understood as having no clusters --- this regime is therefore considered replica symmetric. By contrast, for $\alpha>\acond$, the solution space $\SOL$ is dominated by clusters of size $\exp\{n s_\star\}$, of which there are only a \emph{bounded} number because $\Sigma(s_\star)=0$. Thus there are clusters carrying a non-vanishing fraction of the total mass, so the asymptotic measure has non-trivial clusters and is considered to be genuinely replica symmetry breaking. This scenario persists up to
	\beq\label{e:arsb.abstract}
	\arsb= \sup\bigg\{\alpha : \max_s\Sigma(s)\ge0
		\bigg\}\,,\eeq
which is the $\onersb$ prediction for the satisfiability threshold.
 
\begin{figure}[h!]
\centering
\includegraphics{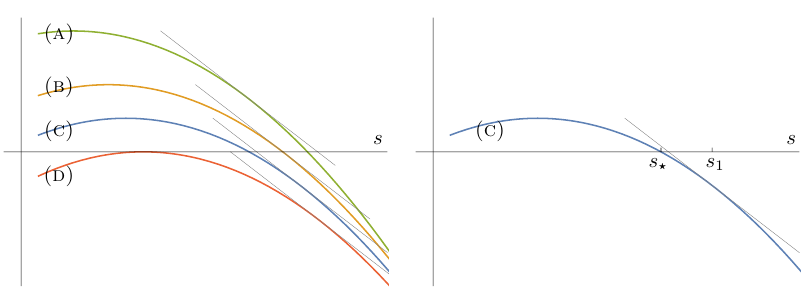}
\caption{The number of clusters of size roughly $\exp\{ns\}$ concentrates around its mean value $\exp\{n\Sigma(s)\}$. The left panel shows $\Sigma(s)\equiv\Sigma(s;\alpha)$ as a function of $s$ for four different values of $\alpha$, together with the tangent lines of slope $-1$. In increasing order of $\alpha$, the curves indicate
\textsc{(a)} $\adrsb<\alpha<\acond$,
\textsc{(b)} $\alpha=\acond$,
\textsc{(c)} $\acond<\alpha<\asat$, and
\textsc{(d)} $\alpha=\asat$.
The right panel shows curve \textsc{(c)} only and indicates the locations of $s_\star$ and $s_1$.}
\label{f:complexity}\end{figure}

We emphasize that the above derivation is highly non-rigorous, relying on unjustified assumptions regarding the measure $\nu$. Nevertheless we have included the above discussion in order to highlight some of the physics intuition. As we discuss in the next section, several of these ideas have had an important role in recent progress on the rigorous study of random \textsc{csp}s. We note also that key aspects of the condensation phenomenon have been rigorously verified in random graph coloring \cite{condensation} and random regular $\nae$ \cite{MR3566764,sly2016number}.

\subsection{Explicit threshold, and sharp upper bound}\label{ss:intro.onersb.threshold}

In the physics perspective, since $\Sigma$ is an explicit function, $\arsb$ is already explicitly characterized by \eqref{e:arsb.abstract}. We now spell this out by making an explicit definition of a function $\Phi(\alpha)$ which corresponds to the physics prediction for $\max_s\Sigma(s;\alpha)$. Throughout what follows, we always assume that $k$ exceeds a large enough absolute constant $k_0$. Further, in view of known bounds (\S\ref{ss:intro.rigorous}) on $\asat$, we restrict attention to the regime
	\beq\label{e:alpha.regime}
	2^k \log 2-2 \equiv
	\albd\le\alpha\le\aubd \equiv 2^k\log2\,.
	\eeq
These assumptions will be made throughout the paper even when not explicitly stated.

Let $\PINT$ denote the space of probability measures on the half-open interval $[0,1)$ --- this means, in particular, that any $\mu\in\PINT$ gives zero measure to the event $\set{\eta=1}$. The interpretation of the measure $\mu$ can be explained roughly as follows (it will be formalized in \S\ref{ss:wp.recursions} below). For an edge $e=(av)$ with literal $\lit_{av}$, we can consider the law of $x_v$ ``in absence of $a$,'' i.e., ignoring the constraint imposed by clause $a$. A random variable $\eta$ sampled from the law $\mu$ (hereafter denoted ``$\eta\sim\mu$'') represents the probability 
``in absence of $a$'' that we have 
$x_v=\minus\lit_{av}$. The randomness of $\eta$ results from the randomness in the neighborhood structure of $v$.

The above interpretation leads naturally to recursive equations for the random variables $\eta$, which are termed ``survey propagation'' equations in the literature (see \cite{MR2351840} and references therein). We express this as a mapping $\Rec\equiv\Rec^\alpha:\PINT\to\PINT$ as follows. Given $\mu\in\PINT$, generate an array of i.i.d.\ samples from $\mu$,
	\beq\label{e:defn.array.ueta}
	\ueta \equiv \bigg(
	(\eta^\plus_{ij},\eta^\minus_{ij})_{i,j\ge1}
	\bigg)\,.\eeq
Let $d^\plus,d^\minus$ be $\Pois(\alpha k/2)$ random variables, independent of $\ueta$ and of one another: $d^\plus$ and $d^\plus$ represent the cardinalities of $\pd v(\plus a)$ and $\pd v(\minus a)$ respectively, where
	{\setlength{\jot}{0pt}\begin{align*}
	\pd v(\plus a)
	&\equiv \set{b\in\pd v\setminus a :
		\lit_{bv}=\lit_{av}}\,,\\
	\pd v(\minus a)
	&\equiv\set{b\in\pd v\setminus a :
		\lit_{bv}=\minus\lit_{av}}\,.
	\end{align*}}%
If $b$ is the $i$-th clause in $\pd v(\plus a)$, the chance ``in absence of $\pd v\setminus b$'' that $x_v$ is forced to equal $\lit_{bv}$ is represented by
	\[
	\hat{u}^\plus_i
	\equiv \prod_{j=1}^{k-1} \eta^\plus_{ij}
	\]
--- this corresponds to the chance ``in absence of $b$'' that $x_u=\minus\lit_{bu}$ for every $u\in\pd b\setminus v$. We define analogously $\hat{u}^\minus_i$ to refer to the $i$-th clause in $\pd v(\minus a)$. The chance ``in absence of $a$'' that none of the clauses in $\pd v(\plus a)$ are forcing to $x_v$ (meaning that the value $x_v=\minus\lit_{av}$ is permitted) is 
	\beq\label{e:Pi.PM}
	\Pi^\plus\equiv \prod_{i=1}^{d^\plus}
	\Big(1-\hat{u}^\plus_i\Big) 
	=\prod_{i=1}^{d^\plus}
	\bigg(1 - \prod_{j=1}^{k-1} \eta^\plus_{ij}
	\bigg)\,.
	\eeq
We can define analogously $\Pi^\minus$ which corresponds to the chance ``in absence of $a$'' that none of the clauses in $\pd v(\minus a)$ are forcing to $x_v$, meaning that the value $x_v=\lit_{av}$ is permitted. We sometimes write $\Pi^\PM\equiv\Pi^\PM(\ud,\ueta)$ 
to emphasize the dependence of $\Pi^\PM$
on the random variables $\ud\equiv(d^\plus,d^\minus)$
and $\ueta$ (from \eqref{e:defn.array.ueta}). Notice that if all the $\eta$'s belong to $[0,1)$, then $\Pi^\PM\in(0,1]$. The chance ``in absence of $a$'' that $x_v$ is not forced by $\pd v(\plus a)$, but is forced by some clause in $\pd v(\minus a)$, is
given by $\Pi^\plus(1-\Pi^\minus)$. However, if $x_v$ is simultaneously forced by both $\pd v(\plus a)$ and $\pd v(\minus a)$, this would invalidate the configuration, since it means that not all the clauses in $\pd v$ can be simultaneously satisfied. The chance ``in absence of $a$'' that $x_v$ is not simultaneously forced by both $\pd v(\plus a)$ and $\pd v(\minus a)$ is given by $\Pi^\plus+\Pi^\minus-\Pi^\plus\Pi^\minus$. We condition on being in a valid configuration simply by taking the ratio of the last two expressions, resulting in 
	\beq\label{e:intro.dist.recurs}
	R(\ud,\ueta)
	\equiv
	\f{\Pi^\plus(1-\Pi^\minus)}
	{ \Pi^\plus + \Pi^\minus - \Pi^\plus \Pi^\minus }\,.
	\eeq
We define $\Rec\mu$ to be the law of $R(\ud,\ueta)$. Since $\Pi^\PM\in(0,1]$, it follows that $R(\ud,\ueta)\in[0,1)$, so $\Rec\mu$ is indeed also an element of $\PINT$. 

\begin{ppn}[proved in \S\ref{ss:pgw.stability}: solution of $\onersb$ distributional recursion]\label{p:fp} Fix $k,\alpha$ and let $\Rec\equiv\Rec^\alpha$ as above. Let $\mu^\ell \equiv \mu^\ell(\alpha) \in\PINT$ ($\ell\ge0$) be the sequence of probability measures defined by $\mu^0=\delta_{1/2}$, and $\mu^\ell=\Rec \mu^{\ell-1}$ for all $\ell\ge1$. For $k\ge k_0$ and $\albd\le\alpha\le\aubd$, this sequence converges weakly as $\ell\to\infty$ to a limit $\mu=\mu^{\infty,\alpha}\in\PINT$, satisfying $\Rec\mu=\mu$.\end{ppn}

The following is the formal characterization of the $\onersb$ prediction $\arsb$ for the $\ksat$ threshold:

\begin{ppn}[proved in \S\ref{s:monotonicity}:
	characterization of $\onersb$ threshold prediction]
\label{p:phi} Given $k,\alpha$, let $\mu=\mu^{\infty,\alpha}$ be the fixed point of Proposition~\ref{p:fp}. Let $\ud$ and $\ueta$ be as above, and let $\ueta'\equiv(\eta_j)_{j\ge1}$ another sequence of i.i.d.\ samples from $\mu$ (independent of $\ud,\ueta$). Let
	\beq\label{e:phi.alpha}
	\Phi(\alpha)
	=\E\bigg[
	\log\f{ \Pi^\plus+\Pi^\minus-
		\Pi^\plus\Pi^\minus }
		{ (1- \prod_{j=1}^k \eta_j
			)^{\alpha(k-1)} }
	\bigg]\,,
	\eeq
where $\E$ indicates the expectation over $(\ud,\ueta,\ueta')$. For $k\ge k_0$, the function $\Phi$ is well-defined and strictly decreasing on the interval $\albd\le\alpha\le\aubd$, with a unique zero $\arsb\equiv\arsb(k)$.
\end{ppn}

Recall that a cluster means, generally, a dense region of the measure, where for us the measure of interest is the uniform measure \eqref{e:unif.sol} over the $\ksat$ solution space $\SOL\subseteq\set{\plus,\minus}^n$. In the regime that we study, with $k\ge k_0$ and $\albd\le\alpha\le\aubd$, it turns out that the clusters are well-separated, so that we can simply \emph{define} a cluster to be a connected component of $\SOL$. (Two assignments $\uBX,\uBX'\in\SOL$ are connected if they differ by a single bit.) With this definition, the above propositions describe the replica symmetric calculation for the \bemph{uniform measure on clusters}. Recall from \S\ref{ss:intro.rsb} that a key step of this calculation is to compute the distribution of a variable incident to a removed clause. In the above, $\eta\in[0,1)$ represents the probability for such a variable to be frozen to the $\minus$ value. The randomness in $\eta$, as described by $\mu$, reflects the random structure of the local neighborhood of this variable.

The normalization for the uniform measure on clusters is the total number $\bm{\Omega}$ of clusters, and $\Phi(\alpha)$ is the associated Bethe free energy. As discussed in \S\ref{ss:intro.rsb}, the measure $\mu$ should satisfy a tree recursion (the map $\Rec$ of Proposition~\ref{p:fp}), and $\Phi$ is expressed in terms of a fixed point for this recursion. The $\onersb$ conjecture for random $\ksat$ says that clusters are \textsc{rs} and so $\Phi(\alpha)$ correctly predicts the free energy, which would mean that $\bm{\Omega}$ concentrates about
	\[\exp\Big\{n\Phi(\alpha)\Big\}
	=\exp\bigg\{n\max_s\Sigma(s;\alpha)\bigg\}\]
(with high probability). This explains why $\arsb$ is defined as the root of $\Phi(\alpha)$.

We emphasize again that the above characterization of $\arsb$ already appears in the physics literature \cite{MR2213115}. To the best of our knowledge, however, it has not been formally proved to be well-defined. Indeed, the proofs of Propositions~\ref{p:fp} and \ref{p:phi} are based on a detailed recursive analysis, which we could not extend to all $k\ge3$. Nevertheless, these propositions do show that $\arsb$ is at least well-defined for $k$ large enough. Having verified this, it is relatively straightforward to deduce the sharp satisfiability upper bound:

\begin{ppn}[proved in \S\ref{ss:threshold.ubd}: $\onersb$ upper bound]
\label{p:ubd} For $k\ge k_0$, random $\ksat$ at $\alpha>\arsb(k)$ is with high probability unsatisfiable. \end{ppn}

Proposition~\ref{p:ubd}
is proved via previously known bounds
\cite{FrLe:03,MR2095932}
for the positive-temperature
version of the $\ksat$ model
 --- that is to say, the measure
	\[
	\nu(\uBX)
	= \f{\exp\{-\beta H(\uBX)\}}{Z(\beta)}
	\]
where $H(\uBX)$ is the number of clauses
violated by $\uBX$. For any $\beta\in[0,\infty)$, it is proved \cite{MR2095932} that 
	\beq\label{e:fl.pt.ubd}
	\f{\E[ \log Z(\beta) ]}{n}
	\le \inf_{\zeta,m}
	\Phi_{\onersb}(\beta,\zeta,m)
	\eeq
where $\zeta$ runs over the space of probability measures over $\mathbb{R}$, $m$ is the $\onersb$ ``Parisi parameter''
which goes over $[0,1]$, and the functional $\Phi_{\onersb}$ is given explicitly in the statement of Theorem~\ref{t:free.energy} below. The proof of the bound \eqref{e:fl.pt.ubd} in \cite{MR2095932} is based on a certain Hamiltonian interpolation scheme, inspired by related results for the SK spin glass \cite{MR1930572,MR1957729,FrLe:03}.

It remains for us to choose $\zeta$ and $m$ to obtain a good upper bound on \eqref{e:fl.pt.ubd}. The $\onersb$ heuristic suggests to choose $\zeta$ in a particular way, such that it is approximately a reparametrization of the measure $\mu$ (from Proposition~\ref{p:fp}). With this choice, we show that as soon as $\alpha$ exceeds $\arsb$, we have (see \eqref{e:neg.free.energy} below)
	\[
	\lim_{\beta\to\infty}\Phi_{\onersb}
		\bigg(\beta,\zeta,\f1{\beta^{1/2}}
		\bigg) <0
	\,.
	\]
This will imply there are no satisfying assignments with high probability, yielding Proposition~\ref{p:ubd}.

\subsection{Sharp lower bound} 
\label{ss:intro.proof.overview}

The main content of this paper is to prove the matching lower bound to Proposition~\ref{p:ubd}. As noted in \S\ref{ss:intro.rigorous}, all recent satisfiability lower bounds, including our current result, are proved by the 
second moment method together with Friedgut's theorem \cite{Friedgut:99}. We now briefly describe the main obstacles to this method, and how they are overcome in our analysis. A more extensive discussion is given in \S\ref{ss:moments}. 

As before, let $Z$ be the total number of $\ksat$ solutions, and let $\E$ denote expectation with respect to $\P=\poisP^{n,\alpha}$. The most basic version of the second moment method would be to prove 
	\[\limsup_{n\to\infty}
	\f{(\E Z)^2}{\E[Z^2]} <\infty\,.\]
The Cauchy--Schwarz inequality then gives 
	\[\liminf_{n\to\infty}\P(Z>0)>0\]
at this
value of $\alpha$; and Friedgut's theorem immediately
implies satisfiability with high
probability at any $\alpha'<\alpha$.

In fact, this basic version of the second moment method fails on random $\ksat$ at \bemph{any} positive clause density --- the ratio $\E[Z^2]/(\E Z)^2$ diverges with $n$ for any positive $\alpha$, including throughout the \textsc{rs} regime. The problem does not go away if we condition on the number of clauses, or make other minor modifications, as we discuss in more detail in \S\ref{ss:moments} below. At a high level, the reason is roughly as follows. Recall (\S\ref{ss:intro.rsb}) that in the \textsc{rs} regime, variables far apart in the graph $\GG$ are nearly independent, and the behavior of each variable depends only on its local neighborhood. In some ``locally homogeneous'' models, either all variables have the same local neighborhood, or there is a variety of local neighborhoods but they all give rise to the same variable behavior.\footnote{See the discussion of ``symmetric'' models in
 \cite[Appx.~A]{Coja-Oghlan:2013:GAK:2488608.2488698}. We use the phrase ``locally homogeneous'' rather than ``symmetric'' to avoid confusion with the (separate) issue of replica symmetry.}
This homogeneity does \bemph{not} hold for random $\ksat$ --- e.g., some variables are incident to more positive literals, and so are more likely to be \textsc{true}. In the \textsc{rs} regime, there is correlation decay
\bemph{conditional} on the graph structure --- but the moment calculation averages over the graph structure, and as a result non-negligible correlations arise. These ``local neighborhood correlations'' cause the second moment method to fail, and are a central difficulty of random $\ksat$.

In spite of this, the second moment method has been successfully applied to lower bound the number of $\ksat$ solutions in the \textsc{rs} regime. In all such results (\cite{1182003,MR2121043,Coja-Oghlan:2013:GAK:2488608.2488698}, see \eqref{e:intro.lbd.rs}), a key step is to make some truncation, or reweighting, such that the resulting model becomes locally homogeneous. The result of \cite{Coja-Oghlan:2013:GAK:2488608.2488698} is notable in that it also conditions on the degree profile of the $\ksat$ instance, an idea which had previously been applied in a simpler model \cite{MR2961553}. This decreases the effect of local neighborhood correlations and gives an improved lower bound.

The best lower bound prior to this work is due to Coja-Oghlan and Panagiotou (\cite{MR3436404}, see~\eqref{e:intro.best.prev}). This advance was especially significant in moving the lower bound past the conjectural condensation threshold of random $\ksat$. Inspired by the $\onersb$ heuristic, the proof of \cite{MR3436404} applies second moment method to the number of solution \bemph{clusters}, rather than the number of individual solutions. This strategy had previously been applied to improve the lower bound for random $\nae$ \cite{MR2961553}, and to obtain sharp satisfiability thresholds in some locally homogeneous models \cite{dss-naesat,MR3440193,MR3689942}. The result of \cite{MR3436404} further incorporates techniques developed in \cite{MR2961553,Coja-Oghlan:2013:GAK:2488608.2488698} for conditioning on the degree profile.

The result of \cite{MR3436404} demonstrates that applying the second moment method to the number of solution clusters, and conditioning on the degree profile, can give very good lower bounds. It became clear, however, that in order to achieve a sharp lower bound, it would be necessary to condition not only on the degree profile, but on the profile of local neighborhood structures to \bemph{arbitrarily large (constant) depth $R$}. The main work of this paper is to carry out this approach: we establish a satisfiability lower bound $\albd(R)$ for each $R$; and show that $\albd(R)\to\arsb$ in the limit $R\to\infty$.

Let us briefly indicate the main difficulties in implementing this strategy. The second moment computation reduces to an optimization problem over a vector $\omega$ of empirical marginals, broken down according to the $R$-neighborhood type --- the dimension of this problem diverges with $R$. The proof of \cite{MR3436404} solves a version of this problem for marginals $\omega$ broken down according to the variable degree. Their analysis relies on an important preprocessing step --- for $k$ large, removing $n\ep_k$ variables with atypical degree leaves behind an nearly regular graph. This allows for very explicit analysis of the second moment, but costs $\ep_k$ in the satisfiability lower bound.

To achieve a sharp lower bound, we can only afford to remove $n\ep_{k,R}$ variables with $\ep_{k,R}\to0$ in the limit $R\to\infty$. Thus we cannot hope to avoid including increasingly pathological vertices as $R$ grows. Instead, we devise a slightly elaborate preprocessing scheme which ensures that bad vertices are surrounded by large buffers of nice vertices. One portion of the paper is occupied with proving that this scheme indeed removes a vanishing fraction $\ep_{k,R}$ of variables.

It remains to solve the second moment optimization problem, where as input we have only rather rough \textit{a~priori} estimates on $\omega$ that are guaranteed by the preprocessing step. The central new idea in this paper is to update $\omega$ in \bemph{blocks} corresponding to trees inside the graph. By keeping the rest of $\omega$ fixed, we can reduce a non-convex optimization problem on large finite graphs to a convex optimization problem on finite trees of bounded (though diverging with $R$) depth, with some fixed boundary conditions. For the tree optimization we make a system of weights that act as Lagrange multipliers for the boundary conditions. The weights are set by an inductive construction, where the preprocessing step was specifically designed to ensure that the weights contract in the desired way. Once these weights are set, it becomes relatively easy to read off the desired second moment bound. This analysis is the main technical contribution of this paper, and may be appliable in the analysis of other models which are not locally homogeneous. We refer to Section~\ref{s:preprocess.defns} for a more detailed proof outline.

\medskip\noindent\textbf{Acknowledgements.} We thank Amir Dembo, Ahmed El Alaoui, Elchanan Mossel, Andrea Montanari, and Lenka Zdeborov\'a for many helpful conversations. We also wish to acknowledge the hospitality of the Theory Group at Microsoft Research Redmond, where part of this work was done. Ahmed El Alaoui and Andrea Montanari reviewed with us a draft of this paper and made innumerable valuable comments, and we especially thank them for their generosity. Finally, we are extremely grateful to the anonymous referee for their careful reading and detailed comments on the paper.

\section{Moment method, cluster encodings, and tree recursions}
\label{s:prelim}

In this section we introduce some of the preliminary formalisms that will be required for the proof. The section is organized as follows: 
\begin{enumerate}[--]
\item In \S\ref{ss:moments} we review the standard first and second moment calculations for random $\ksat$, and show that the second moment method fails in this model.

\item In \S\ref{ss:local.inhomogeneity.and.rsb}
we elaborate on two reasons for the failure of the second moment method:
the local inhomogeneity of the random $\ksat$ graph,
 and the phenomenon of large clusters (replica symmetry breaking) in the solution space.

\item In \S\ref{ss:frozen.model.and.coarsening} we introduce a combinatorial encoding of clusters, the so-called ``frozen model,'' which will be the basis of our modified moment method approach.

\item In \S\ref{ss:wp.and.color.models} we introduce two more combinatorial encodings, the warning propagation model and the color model. They are equivalent to the frozen model, but each model has its own drawbacks and advantages.

\item In \S\ref{ss:wp.recursions} we introduce tree recursions for the warning propagation model.
\item In \S\ref{ss:bp} we introduce weighted versions of the color model, and review the belief propagation (\textsc{bp}) equations in this context. We show that the tree recursions of \S\ref{ss:wp.and.color.models} can be recovered as a special case.

\end{enumerate}
Before proceeding further, we first review the formal definition of the model.

\begin{dfn}[bipartite factor graph]
\label{d:vfe}
A \bemph{bipartite factor graph} is a graph $\GG=(V,F,E)$ whose vertex set $V\cup F$ is partitioned into \bemph{variables} $V$ and \bemph{clauses} $F$,
with undirected edges $E$ joining variables to clauses.
We generically denote variables $u,v,w$, clauses $a,b,c$, and edges $e$. For each edge $e$, we write $a(e)$ for the incident clause and $v(e)$ for the incident variable. Each $e$ comes with a \bemph{sign} $\lit_e$, indicating whether the inclusion of variable $v(e)$ in clause $a(e)$ is positive ($\lit_e=\plus$) or negative ($\lit_e=\minus$). We allow for multi-edges, so $\GG$ might have two edges $e,e'$ joining $a$ to $v$ (possibly with $\lit_e\ne\lit_{e'})$). If there is a unique edge $e$ joining clause $a$ to variable $v$, we will sometimes denote it as $e=(av)=(va)$, and write $\lit_e\equiv\lit_{av}$. For any vertex $x\in V\cup F$, we write $\pd x$ for its neighboring vertices. Similarly we write $\delta x$ for the edges incident to $x$. For each clause $a\in F$ we regard $\pd a$ and $\delta a$ as \bemph{ordered tuples}. For each edge $e\in E$ we write $j(e)$ to indicate the position of $e$ in $\delta a(e)$, so $j(e)\in[k]$. For each variable $v\in V$ we partition its neighbors and incident edges according to edge sign:
	{\setlength{\jot}{0pt}
	\begin{alignat}{2}
	\nonumber
	\delta v(\plus)
		&\equiv\set{e\in\delta v
			: \lit_e=\plus}\,,\quad
	& \pd v(\plus)
		&\equiv\set{a(e) : e\in\delta v(\plus)}\,,\\
	\delta v(\minus)
		&\equiv\set{e\in\delta v
			: \lit_e=\minus}\,,\quad
	&\pd v(\minus) 
		&\equiv\set{a(e) : e\in\delta v(\minus)}\,.
	\label{e:var.signed.half.edges}
	\end{alignat}}%
(In scenarios with multi-edges, the sets
$\pd v(\plus)$ and $\pd v(\minus)$ can intersect, and are not in one-to-one correspondence with the sets $\delta v(\PM)$. For this reason, we always work with $\delta v(\PM)$ to avoid ambiguity, except in cases where multi-edges are expressly prohibited.) 
We will also refer to $\GG$ as a \bemph{\textsc{sat} problem instance}. Furthermore we call $\GG$ a \bemph{$\ksat$ problem instance} if each clause $a\in F$ has width $|\delta a|=k$.
\end{dfn}

\begin{dfn}[satisfying assignments]\label{d:formal.dfn.sat}
If $\GG=(V,F,E)$ is a bipartite factor graph as in Definition~\ref{d:vfe}, it defines a mapping
$\GG : \set{\plus,\minus}^V\to \set{0,1}$
as follows:
for any \bemph{variable assignment}
$\uBX\in\set{\plus,\minus}^V$,
	\[
	\GG(\uBX) 
	\equiv\prod_{a\in F}
		\bigg\{1 - \prod_{e\in\delta a}
			\f{1-\lit_e \BX_{v(e)}}{2}\bigg\}\,.
	\]
The instance $\GG$ is called \bemph{satisfiable} if and only if its set
	\[\SOL(\GG)
	\equiv\GG^{-1}(1)
	\equiv\bigg\{\ux\in\set{\plus,\minus}^V
		:\GG(\ux)=1\bigg\}
	\subseteq\set{\plus,\minus}^V\]
of \bemph{satisfying assignments} is nonempty.
\end{dfn}

\begin{dfn}[random $\ksat$]\label{d:formal.random.ksat} To generate an instance of the \bemph{random $\ksat$ model at clause density $\alpha$}, start with $n$ labelled variables $V = [n] \equiv \set{1,\ldots,n}$, as well as $M$ labelled clauses $F = [M]\equiv\set{1,\ldots,M}$, where $M$ is sampled from the $\Pois(n\alpha)$ distribution. Independently for each clause $a\in F$, sample its $k$-tuple of variables $\pd a$ uniformly at random from $V^k=[n]^k$, then sample the $k$-tuple of signs $(\lit_e)_{e\in\delta a}$ uniformly at random from $\set{\plus,\minus}^k$. This defines a family of probability measures $\poisP\equiv\poisP^{n,\alpha}$ over $\ksat$ instances, indexed by $n$ and parametrized by the expected clause density $\alpha$. Writing $\P_{n,m}$ for the measure $\poisP$ conditioned on $M=m$, we can decompose
	\beq\label{e:poisson.erdos.renyi}
	\poisP \equiv \poisP^{n,\alpha}
	= \sum_{m\ge0} \POIS_{n\alpha}(m)\P_{n,m}\,,\quad
	\POIS_\lm(m) \equiv \f{e^{-\lm}\lm^m}{m!}\,.\eeq
Note that the conditional measure $\P_{n,m}$ does not depend on $\alpha$.
\end{dfn}

\begin{rmk}\label{r:distance} A \textsc{sat} problem instance can be equivalently viewed as a hypergraph, with vertices and hyperedges corresponding to variables and clauses respectively. Each hyperedge should be viewed as an ordered tuple $(v_1,\ldots,v_k)$ of variables, with corresponding signs $(\lit_1,\ldots,\lit_k)$. A $\ksat$ instance thus corresponds to a $k$-uniform hypergraph. If $\GG=(V,F,E)$ is the bipartite factor graph representation of a \textsc{sat} instance, we assign length $1/2$ to all its edges, so that graph distances in $\GG$ will be consistent with the standard graph distances in the hypergraph representation. For any vertex $x$ in $\GG$ we will write $\pd_s x$ for the set of vertices at distance $s$ from $x$. We will write $N(x)\equiv \pd_1(x)$. If $v$ is a variable then $N(v)$ is the set of variables sharing a clause with $v$, and we often refer to these as the ``neighboring variables of $v$.''
\end{rmk}

\subsection{Moments of satisfying assignments}
\label{ss:moments}

We now review the standard first and second moment calculations for random $\ksat$. The purpose of this discussion is to illustrate the main obstructions to proving sharp bounds on the satisfiability threshold. These issues were known prior to our work, and we refer to further detailed discussions in the existing literature \cite{1182003,MR2121043}.

For comparison, we will present the moment calculations for random $\ksat$ as well as a closely related model, 
\bemph{random $\knae$}, which has also been extensively studied (notably by \cite{1182003,MR2961553}). On a bipartite factor graph $\GG=(V,F,E)$, a variable assignment $\ux\in\set{\plus,\minus}^V$ is called an \bemph{not-all-equal-\textsc{sat} ($\nae$) assignment} if both $\ux$ and $\minus\ux$ are valid \textsc{sat} assignments. Thus, while the \textsc{sat} assignments of $\GG$ are given by $\SOL(\GG)\equiv\GG^{-1}(1)$, the $\nae$ assignments are given by $\NAE(\GG)\equiv \SOL(\GG)\cap[-\SOL(\GG)]$. We define the corresponding \bemph{partition functions}
	\begin{align*}
	Z\equiv Z(\GG)
	&\equiv \Big|\SOL(\GG)\Big|
	\equiv\bigg|\bigg\{
	\textup{\textsc{sat} assignments of $\GG$}
	\bigg\}\bigg|\,,\\
	\ZNAE\equiv\ZNAE(\GG)
	&\equiv
		\Big|\NAE(\GG)\Big|
	\equiv
	\bigg|\bigg\{
	\textup{\textsc{nae-sat} assignments of $\GG$}
	\bigg\}\bigg|\,.
	\end{align*}
Clearly, $\ZNAE(\GG) \le Z(\GG)$ for any instance $\GG$.

For $M\sim\Pois(n\alpha)$, we have
 $|M-n\alpha|\le n^{1/2}\log n$ with high probability. We can then see from \eqref{e:poisson.erdos.renyi} that in order to show $\poisP^{n,\alpha}(E_n)=o_n(1)$ for some event $E_n$, it is sufficient to show $\P_{n,m}(E_n)=o_n(1)$ uniformly over all $m$ satisfying $|m-n\alpha| \le n^{1/2}\log n$. We will compute first and second moments of $Z(\GG)$ and $\ZNAE(\GG)$ for $\GG$ distributed according to the conditional measure $\P_{n,m}$. The reason to work with $\P_{n,m}$ rather than $\poisP^{n,\alpha}$ is that the fluctuations in $M$ alone are already enough to make the second moment method fail under $\poisP^{n,\alpha}$.
Fixing the number of clauses is an easy way to remove some variance from the second moment calculation, and will allow us to see the more difficult sources of variance that remain after $m$ is fixed. The second moment also fails under $\P_{n,m}$ for reasons that are more subtle and that determine the proof strategy.

For ease of exposition, for the current discussion
we will assume that assume $n\alpha$ is an integer, and consider $\P_{n,m}$ only for $m=n\alpha$.\footnote{Strictly speaking, we should fix $\alpha$ and consider $m=n\alpha'$ for all $|n(\alpha'-\alpha)|\le n^{1/2}\log n$. However the calculations that follow are not very sensitive to slight perturbations in $\alpha$: it the moment method gives satisfiability (or unsatisfiability) with high probability under $\P_{n,n\alpha}$, then it also gives the same result under $\P_{n,n\alpha'}$ for $|n(\alpha'-\alpha)|\le n^{1/2}\log n$, unless $\alpha$ is exactly at a threshold. For this reason we prefer to keep the notation simple in this introductory discussion, and consider only $\alpha=\alpha'$.} Let $\E_{n,n\alpha}$ denote expectation with respect to $\P_{n,n\alpha}$. We then have
	\begin{align}
	\label{e:first.mmt.sat}
	\E_{n,n\alpha} Z
	&=2^n \bigg(1-\f1{2^k}\bigg)^m
	= \exp\bigg\{n\bigg[
	\log2+\alpha\log\bigg(1-\f1{2^k}\bigg)
	\bigg]
	\bigg\}
	\equiv\exp\{n\ff_\textsc{sat}(\alpha)\}
	\,,\\
	\E_{n,n\alpha} \ZNAE
	&=2^n \bigg(1-\f2{2^k}\bigg)^m
	= \exp\bigg\{n
	\bigg[
	\log2+\alpha\log\bigg(1-\f2{2^k}\bigg)
	\bigg]
	\bigg\}
	\equiv\exp\{n\ff_\textsc{nae}(\alpha)\}\,.
	\nonumber
	\end{align}
Write $\alpha_1$ for the solution of $\ff_\textsc{sat}(\alpha)=0$; this is the \bemph{first moment threshold} for random $\ksat$. For $\alpha>\alpha_1$ the expected value of $Z$ is exponentially small in $n$, so (by Markov's inequality) it holds with high probability that $Z$ is zero, meaning that the instance is unsatisfiable. 
Thus the satisfiability transition for random $\ksat$
is upper bounded by the first moment threshold,
which occurs just below $2^k\log2$.
The analogous statement holds for random $\knae$: the satisfiability transition is upper bounded by the solution of $\ff_\textsc{nae}(\alpha)=0$, which occurs just before $2^{k-1}\log2$.


To lower bound the satisfiability threshold, one approach is to apply the \bemph{second moment method}, based on the following consequence of the Cauchy--Schwarz inequality:
	\beq\label{e:paley.weiner}
	\f{(\E Z)^2}{\E[Z^2]}
	=\f{(\E[ Z\Ind{Z>0} ])^2}{\E[Z^2]}
	\le \P(Z>0)\,.\eeq
In the most naive application, one could take $\E$ to be expectation with respect to the overall measure $\poisP^{n,\alpha}$, and try to show a bound of the form
	\beq\label{e:most.naive.second.moment}
	\limsup_{n\to\infty}
	\f{\E (Z^2)}{(\E Z)^2}
	\le C(k,\alpha)<\infty\,,
	\eeq
for any positive $\alpha$. In fact, as we already alluded to, the bound \eqref{e:most.naive.second.moment} is false for all positive $\alpha$: indeed, one can use the above calculation \eqref{e:first.mmt.sat} to see that
for all $|n(\alpha'-\alpha)|\le n^{1/2}\log n$, we have
	\beq\label{e:comparing.second.mmt.if.cond.on.m}
	\E_{n,n\alpha'}Z
	= \exp\{n \ff_\textsc{sat}(\alpha') \}
	\ll
	\sum_{m\ge0}
	\POIS_{n\alpha}(m)\E_{n,m} Z
	=\E Z\,,
	\eeq
with $\POIS_{n\alpha}$ as defined by \eqref{e:poisson.erdos.renyi}. 
It follows by Markov's inequality that $Z \ll \E Z$ with high probability. This implies that \eqref{e:most.naive.second.moment} must be false, otherwise
we would have a contradiction to \eqref{e:paley.weiner}.
By a similar calculation, the bound \eqref{e:most.naive.second.moment}
also fails with $\ZNAE$ in place of $Z$.

A more promising approach is to try to establish a conditional second moment bound, of the form
	\beq\label{e:second.mmt.given.m}
	\limsup_{n\to\infty}\bigg[
	\sup\bigg\{\f{\E_{n,m}(Z^2)}{(\E_{n,m} Z)^2} 
	: |m-n\alpha| \le n^{1/2}\log n
	\bigg\} \bigg]
	\le C(k,\alpha) <\infty\,.\eeq
If \eqref{e:second.mmt.given.m} were to hold
at some positive clause density $\alpha$, then substituting it into \eqref{e:paley.weiner} would show that satisfiability occurs with asymptotically positive probability under $\poisP^{n,\alpha}$:
	\[\liminf_{n\to\infty}
	\poisP^{n,\alpha}(\textsc{sat})
	\ge
	\liminf_{n\to\infty}\bigg[
	\inf\bigg\{
	\P_{n,m}(\textsc{sat})
	: |m-n\alpha| \le n^{1/2}\log n
	\bigg\}\bigg]
	\ge
	\f1{C(k,\alpha)}>0\,.
	\]
Then Friedgut's theorem~\eqref{e:friedgut}
immediately gives satisfiability with high probability at any $\alpha-\ep<\alpha$, which would imply a satisfiability lower bound, $\asat \ge \alpha$. 

In fact, we will see that the bound \eqref{e:second.mmt.given.m} also fails for random $\ksat$ at all positive $\alpha$, although it gives a non-trivial lower bound for the random $\knae$ model (\cite{1182003}, and reviewed below). To see this, we decompose
	\beq\label{e:pairs.sat.assignments.overlap}
	Z^2
	= \sum_z Z^2[z]\,,\quad
	Z^2[z]
	\equiv \Bigg| \Bigg\{ \hspace{-3pt}
	\begin{array}{c}
	 \text{pairs }
	(\uBX^1,\uBX^2) \in
		\SOL(\GG)\times\SOL(\GG)
		\\
	\text{with $|\set{v\in V :x^1_v=x^2_v }|=nz$}
		\end{array}
	\hspace{-3pt} \Bigg\} \Bigg|,
	\eeq
where the sum is over $z=j/n$ for integer $0\le j\le n$. For example,
\begin{enumerate}[(i)]
\item The value $z=1$ corresponds to identical configurations $\ux^1=\ux^2$, and $Z^2[1]=Z$;

\item The value $z=0$ corresponds to antipodal configurations $\ux^1=\minus\ux^2$, and $Z^2[0]=\ZNAE$;

\item The value $z=1/2$ corresponds to configurations $\ux^1,\ux^2$ that ``look independent'' in the sense that if $v$ is a uniformly random variable, knowing $x^1_v$ does not give any information about $x^2_v$.

\end{enumerate}
For each $z$ we calculate the corresponding second moment contribution to be
	\[
	\E_{n,n\alpha}[Z^2[z]]
	= 2^n \binom{n}{nz}
	\bigg( 1 - \f{2}{2^k} + \bigg(\f{z}{2}\bigg)^k \bigg)^{n\alpha}
	= \f{\exp\{ n
	\ff_{\textsc{sat},2}(z;\alpha) \}}
	{ n^{O(1)}}\,,\]
where the exponent $\ff_{\textsc{sat},2}(z;\alpha)$
can be derived using Stirling's formula,
and does not depend on $n$:
	\beq\label{e:basic.second.moment.z}
	\ff_{\textsc{sat},2}(z;\alpha)\equiv
	\underbrace{
	\bigg[\log2 + \Ent(z)
	\bigg]
	}_{\text{entropy term}}
	+
	\underbrace{\alpha\log\bigg( 1 - \f{2}{2^k}
		+ \bigg(\f{z}{2}\bigg)^k 
		 \bigg)}
		_{\text{probability term}}
	\eeq
where $\Ent(z)=-z\log z-(1-z)\log(1-z)$ denotes the standard entropy function. For comparison, 
	\[\ff_{\textsc{nae},2}(z;\alpha)\equiv
		\underbrace{\bigg[
		\log2+\Ent(z)\bigg]}_{\text{entropy term}}
		+ 
		\underbrace{\alpha\log\bigg(
		1 - \f{4}{2^k}
			+ \f{(z^k + (1-z)^k) 2}{2^k}
		\bigg)}_{\text{probability term}}\]
gives the corresponding exponent for 
$\E[(\ZNAE)^2[z]]$.

By comparing \eqref{e:first.mmt.sat} with \eqref{e:basic.second.moment.z}, we find that
$\ff_\textsc{sat}(1/2;\alpha)-2\ff_\textsc{sat}(\alpha)$ is exactly zero. However, one can notice in \eqref{e:basic.second.moment.z}
that the entropy term is maximized at $z=1/2$, but for any positive $\alpha$ the probability term is strictly increasing with $z$. This will mean that 
$\ff_{\textsc{sat},2}(z;\alpha)$
has strictly positive derivative
with respect to $z$ at $z=1/2$, i.e., $z=1/2$
is not the maximizer for $\ff_{\textsc{sat},2}(z;\alpha)$. This will imply that the ratio
	\[
	\f{\E_{n,n\alpha}[Z^2]}
	{(\E_{n,n\alpha} Z)^2}
	= n^{O(1)}
	\exp \bigg\{ n \sup\bigg\{
	\ff_{\textsc{sat},2}(z;\alpha)
	-2\ff_\textsc{sat}(\alpha)
	: 0\le z\le 1
	\bigg\}\bigg\}
	\]
grows exponentially with $n$, which is in contradiction to \eqref{e:second.mmt.given.m}. By contrast, the function $\ff_{\textsc{nae},2}$ is \emph{stationary} at $z=1/2$ at any fixed $\alpha$, in fact attains its global maximum at $z=1/2$ for a non-trivial range of $\alpha$; see Figure~\ref{f:overlap}. This observation was used in previous work (\cite{1182003}, see \eqref{e:intro.lbd.rs}) to lower bound the satisfiability threshold for random $\knae$. 

\begin{figure}[h!]
\centering
\begin{subfigure}[b]{0.47\textwidth}
\centering
\includegraphics[width=\textwidth]{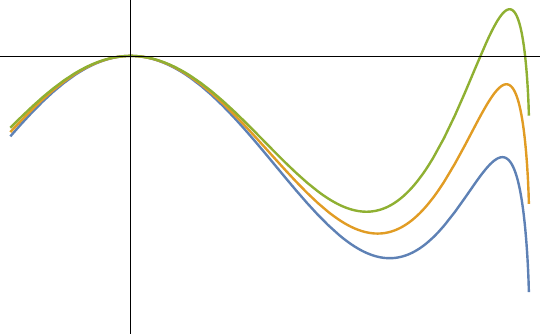}
\caption{$z\mapsto\ff_{\textsc{nae},2}(z;\alpha)
	-2\ff_\textsc{nae}(\alpha)$)}
\label{f:overlap.nae}
\end{subfigure}
\quad
\begin{subfigure}[b]{0.47\textwidth}
\centering
\includegraphics[width=\textwidth]{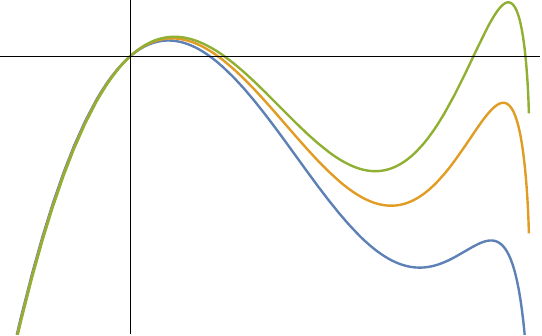}
\caption{$z\mapsto\ff_{\textsc{sat},2}(z;\alpha)
-2\ff_\textsc{sat}(\alpha)$}
\label{f:overlap.sat}
\end{subfigure}
\caption{Comparison of second moment
	with first moment squared for
	random $\nae$
	and \textsc{sat} \\
	for $k=6$
	(with same qualitative phenomena occuring
	for all $k$). \\
	In each figure,
	the horizontal axis is at zero
	while the vertical axis
	is at $z=1/2$.\\
	Each curve corresponds to a different value of $\alpha$, with the uppermost curve in the\\
	right panel corresponding
	to the numerically computed value of $\arsb(k)$
	\cite{MR2213115}.}
\label{f:overlap}
\end{figure}

\subsection{Local inhomogeneity and replica symmetry breaking}
\label{ss:local.inhomogeneity.and.rsb}

The calculation of \S\ref{ss:moments} demonstrates two distinct (though entangled) issues, which we already mentioned in Section~\ref{s:intro} --- \eqref{item:homog}~lack of ``local homogeneity'' and \eqref{item:clust}~large solution clusters and ``replica symmetry breaking.'' These issues manifest themselves in the above calculation roughly as follows:
\begin{enumerate}[(I)]
\item \label{item:homog} For all positive $\alpha$, the point $z=1/2$ is not a local maximizer of $\ff_{\textsc{sat},2}(z;\alpha)$;
\item \label{item:clust} For all $\alpha\in(\alpha_2,\alpha_1)$ (where $\alpha_2<\asat<\alpha_1$), the function $\ff_{\textsc{sat},2}(z;\alpha)$ has another local maximizer $z'$ slightly below $1$, such that $\ff_{\textsc{sat},2}(z';\alpha)>2\ff_\textsc{sat}(\alpha)$.
\end{enumerate}
Of course, the desired second moment bound \eqref{e:second.mmt.given.m} cannot succeed in the presence of either \eqref{item:homog} or \eqref{item:clust}. Point~\eqref{item:clust} reflects the fact that the 
second moment can be dominated by an exponentially rare event where there is an unusually large number of pairs of nearby solutions ($z$ near one), i.e., there is an \bemph{atypically large cluster} of solutions. For both random $\ksat$ and random $\knae$, the $\alpha_2$ of point~\eqref{item:clust} occurs strictly below the satisfiability threshold; in fact, it occurs just below the condensation threshold $\acond$ that marks the onset of \bemph{replica symmetry breaking} (see Figure~\ref{f:phase}). For an extensive discussion of this issue, we refer the reader to two works \cite{dss-naesat, sly2016number} on random \bemph{regular} $\knae$: this is a simplified model where the first issue~\eqref{item:homog} does not arise at all, and as a result there is a very precise correspondence between replica symmetry breaking and problems in the moment method. 

We next turn our attention to point~\eqref{item:homog}. The reflects that the moment calculation favors \bemph{graphs} that are slightly rare, for which the solution set is unusually large. The phenomenon results from the inherent asymmetry between $\plus$ and $\minus$ in the \textsc{sat} predicate, which is absent from \textsc{nae-sat}. For instance, it is reasonable to expect that to have more \textsc{sat} assignments, it is favorable to have atypically many edges all take the same sign. To be more explicit, let $\DD$ be the empirical degree profile of $\GG$, 
	\beq\label{e:ordinary.deg.profile}
	\DD(d^\plus,d^\minus)
	= \f1n
	\bigg|\bigg\{v\in V : 
	|\delta v(\plus)|=d^\plus \textup{ and }
	|\delta v(\minus)|=d^\minus
	\bigg\}\bigg|\,.\eeq
Under $\P_{n,n\alpha}$, we expect that the random profile $\DD$ is concentrated around the typical profile $\DDtyp=(\POIS_{\alpha k/2})^{\otimes 2}$
 (using the notation of \eqref{e:poisson.erdos.renyi}), with gaussian fluctuations: that is to say, we expect
	\beq\label{e:intro.probab.quadratic}
	\f{\P_{n,n\alpha}(\DD)}
		{\P_{n,n\alpha}(\DDtyp)}
	\le
	\exp\bigg\{ -n 
	\Big(\DD-\DDtyp\Big)^\st
	\bm{K}
	\Big(\DD-\DDtyp\Big)
	\bigg\}
	\eeq
for some fixed positive semi-definite $\bm{K}$. On the other hand, we also expect
	\beq\label{e:intro.expect.linear}
	\f{\E_{n,n\alpha}(Z\,|\,\DD)}
		{\E_{n,n\alpha}(Z\,|\,\DDtyp)}
	\ge
	\exp\bigg\{
	n \Big(\bm{u},\DD-\DDtyp\Big)
	\bigg\}
	\eeq
for some fixed vector $\bm{u}\ne0$.
In particular, $(\bm{u},\DD)$ might count
the fraction of $\plus$ edges in $\DD$:
the typical fraction is $(\bm{u},\DDtyp)=1/2$,
but we would expect that
$\E_{n,n\alpha}(Z\,|\,\DD)$ could be made larger by taking
$(\bm{u},\DD)=1/2+\delta$ for small positive $\delta$.
This is to say that
$\DD=\DDtyp$ is optimal for \eqref{e:intro.probab.quadratic}
but not for \eqref{e:intro.expect.linear},
so it is not optimal for the product
of \eqref{e:intro.probab.quadratic}~and~\eqref{e:intro.expect.linear}. It follows that
for all $\DD'$ sufficiently close to $\DDtyp$ we have
	\beq\label{e:compare.second.mmt.around.DDtyp}
	\E_{n,n\alpha}(Z\,|\,\DD')\ll
	\sum_{\DD}
		\P_{n,n\alpha}(\DD)
		\,\E_{n,n\alpha}(Z\,|\,\DD)
	=\E_{n,n\alpha} Z\,.
	\eeq
This can be viewed as a (slightly more complicated) analogue of \eqref{e:comparing.second.mmt.if.cond.on.m}.
It implies $Z\ll \E_{n,n\alpha} Z$ with high probability.
As a result \eqref{e:second.mmt.given.m} must fail, otherwise \eqref{e:paley.weiner} would be contradicted.
On the other hand, for random $\knae$, \eqref{item:homog} does not occur, and the second moment method succeeds 
for some range of positive $\alpha$ (until \eqref{item:clust} arises). From this we can conclude that $\DDtyp$ is optimal for $\E_{n,n\alpha}(\ZNAE\,|\,\DDtyp)$.

Another point of view (which is really another side of the same coin) is that \eqref{item:homog} reflects the ``local inhomogeneity'' of \textsc{sat} solutions. Conditional on any instance $\GG$, let $\mu_{\GG}$ be the uniform measure over the set of satisfying assignments $\SOL(\GG)$ (assuming that it is nonempty). We emphasize that $\mu_{\GG}$ is a \bemph{random measure}, since $\SOL(\GG)$ is a random set. We use $\langle \cdot \rangle_{\GG}$ to denote averaging with respect to $\mu_{\GG}$:
	\[\langle f \rangle_{\GG}
	\equiv \sum_{\ux\in\set{\plus,\minus}^n}
	\mu_{\GG}(\ux)f(\ux)
	= \f1{|\SOL(\GG)|}
	\sum_{\ux\in\SOL(\GG)}f(\ux)\,.\]
We will write $[\cdot]_{\GG}$ for the average with respect to the uniform measure over the $\nae$ solutions $\NAE(\GG)$. We say that \textsc{sat} solutions are \bemph{locally inhomogeneous} because on general instances $\GG$, the variable mean $\langle x_v \rangle_{\GG}$ is not constant over $v\in V$, since we expect it to depend for instance on $|\delta v(\plus)|-|\delta v(\minus)|$. This is in contrast with the $\nae$ model, where the symmetry $\NAE(\GG)=\minus\NAE(\GG)$ implies $[x_v]_{\GG}=0$ for all $v\in V$, so that the model enjoys \bemph{local homogeneity}. This property affects the moment calculation in the following way. Given $\GG$, form a new graph $\GG'$ as follows: add a new clause $a$, connect $a$ to a uniformly random $k$-tuple of existing variables in $\GG$, and sample uniformly random signs $(\lit_e)_{e\in\delta a}$. Then $\SOL(\GG')$ is a subset of $\SOL(\GG)$, and
	\[
	\f{Z(\GG')}{Z(\GG)}
	= 
	\sum_{\ux\in\SOL(\GG)}
	\f{\Ind{\ux
		\textup{ satisfies } a
		}}{|\SOL(\GG)|}
	=1-\mu_{\GG}\bigg(
	\lit_e x_{v(e)}=\minus
	\textup{ for all }e\in\delta a
	\bigg)\,.
	\]
The variables $v(e)$ ($e\in\delta a$) are typically far apart from one another in the original graph $\GG$. We will now make the simplifying assumption that they are \bemph{approximately independent} under $\mu_{\GG}$. This assumption is not rigorous, but for $\alpha\le\acond$ it is not unreasonable, since it matches
the physics prediction that $\mu_{\GG}$ exhibits some form of \bemph{correlation decay}. The assumption allows us to simplify the above as
	\beq\label{e:as2}
	\f{Z(\GG')}{Z(\GG)}
	\doteq 1-\prod_{e\in\delta a} \mu_{\GG}(
		\lit_e x_{v(e)}=\minus)
	=1-\prod_{e\in\delta a}
	\f{1-\lit_e\langle x_{v(e)} \rangle_{\GG} }{2}\,.
	\eeq
The key point is that, conditional on $\GG$,
the ratio \eqref{e:as2} is a \bemph{nondegenerate random variable}, due to the randomness in the literals $\lit_e$
and in the choice of the variables $v(e)$ ($e\in\delta a$). By contrast, under the same assumptions,
	\[\f{\ZNAE(\GG')}{\ZNAE(\GG)}
	=1
	-\prod_{e\in\delta a}\f{1-\lit_e[x_{v(e)}]_{\GG} }{2}
	-\prod_{e\in\delta a}\f{1+\lit_e[x_{v(e)}]_{\GG} }{2}
	= 1-\f{2}{2^k}\,,
	\]
a \bemph{deterministic constant}, simply because $[x_v]_{\GG}$ is constant. However, one can imagine building the entire graph $\GG=\GG_n\sim\P_{n,n\alpha}$ by a sequence $(\GG_0,\ldots,\GG_n)$ where each $\GG_i$ is roughly distributed according to $\P_{i,i\alpha}$, and is formed from $\GG_{i-1}$ by a small number of random local changes as above: adding a variable, and adding or deleting a small number of clauses. We can then represent $Z(\GG)$ as a telescoping product
	\[
	Z(\GG)
	\doteq
	Z(\GG_0)
	\prod_{i=1}^n
	\f{Z(\GG_i)}{Z(\GG_{i-1})}\,,
	\]
where we expect the $n$ terms in the product to be roughly independent from one another, and we think of each term
as being analogous to \eqref{e:as2}. This would suggest that, under $\P_{n,\alpha}$, the variance of $\log Z(\GG)$ is of order $n$, while the variance of $\log \ZNAE(\GG)$ is small. This is consistent with the fact that $Z \ll \E_{n,n\alpha} Z$ with high probability for all positive $\alpha$, while $\ZNAE$ concentrates around $\E_{n,n\alpha} \ZNAE$ if $\alpha$ is not too large. 

To prove an exact satisfiability lower bound by the second moment method, it is necessary to address both issues of
local inhomogeneity~\eqref{item:homog} and replica symmetry breaking~\eqref{item:clust}. For~\eqref{item:clust}, the idea is to count solution clusters instead of individual solutions. On a given instance $\GG$, recall that $\bm{\Omega}\equiv\bm{\Omega}(\GG)$ denotes the number of \textsc{sat} solution clusters (connected components of $\SOL(\GG)$), and let $\bm{\Omega}_\textsc{nae}\equiv\bm{\Omega}_\textsc{nae}(\GG)$ denote the number of $\nae$ solution clusters (connected components of $\NAE(\GG)$). Since the random variables $\bm{\Omega}$ 
and $\bm{\Omega}_\textsc{nae}$ give unit weight to each cluster regardless of cluster size, their moments are not affected by atypically large clusters. This approach is implemented in several works on models that are locally homogeneous \cite{dss-naesat,MR3689942,sly2016number}. An important technical ingredient is a combinatorial representation of clusters as elements
$\ux\in\set{\plus,\minus,\free}^n$, which will be explained in the remainder of this section.
We let $\CLUSTERS(\GG)$ be the set of clusters of $\GG$,
regarded as a subset of $\set{\plus,\minus,\free}^n$.

We now define $\nu_{\GG}$ to be the uniform measure on
$\CLUSTERS(\GG)$. Thus $\nu_{\GG}$
is a random measure over 
$\set{\plus,\minus,\free}^n$.
 Problem~\eqref{item:homog} in this context is that \textsc{sat} clusters are not locally homogeneous, in that $\nu_v$ (the marginal law of $x_v$ under $\nu_{\GG}$) is not constant over $v\in V$. In fact, although we saw that $\nae$ solutions are locally homogeneous, it turns out that $\nae$ clusters are not: in the $\set{\plus,\minus,\free}^n$ representation, a variable is more likely to be $\free$ if it has low degree. (This problem goes away in the random regular $\knae$ model where all variables have the same degree.) To address this problem, it is natural to consider conditioning on a degree profile $\DD$ that is close to the typical one $\DDtyp$. This takes care of \eqref{e:compare.second.mmt.around.DDtyp}. It also makes the model more locally homogeneous in the sense that conditioning on $\DD$ partitions the variables $v$ into different classes according to their $\PM$ degrees $(|\delta v(\plus)|,|\delta v(\minus)|)$; and the fluctuations of $\nu_v$ \bemph{within each class} are smaller than the fluctuations of $\nu_v$ over all $v\in V$.

However, as we already suggested in \S\ref{ss:intro.proof.overview}, conditioning on $\DD$ alone does not resolve the problem, because $\nu_v$ depends on much more than the $\PM$ degree of $v$: we expect it to depend on the entire \bemph{local neighborhood structure} of $\GG$ near $v$. This motivates the following definitions Let $B_R(v)$ be the $R$-neighborhood of $v$, which we regard as a graph rooted at $v$. Let $\DD_R\equiv\DD_R(\GG)$ be the probability measure on rooted graphs defined by
	\[\DD_R(T)= \f{|\set{v\in V:B_R(v) \cong T}| }{ |V|}\,,\]
where $\cong$ denotes rooted graph isomorphism. We regard this $\DD_R$ as a \bemph{generalized degree profile}, and note that the ordinary degree profile $\DD$ of \eqref{e:ordinary.deg.profile} coincides with $\DD_R$ for $R=1/2$ (recalling Remark~\ref{r:distance}). With high probability under $\P=\P_{n,n\alpha}$, the measure $\DD_R$ lies within $o_n(1)$ total variation distance of a measure $\DD_{\star,R}$, which is the law of the first $R$ levels of a certain Poisson Galton--Watson tree (\S\ref{ss:dist.rec}). Conditioning on $\DD_R$ partitions the variables $v$ into different classes according to the structures of their local neighborhoods $B_R(v)$. The expectation is that $\nu_v$ becomes constant within each class in the limit $R\to\infty$. However, conditional on $\DD_R$ for large fixed $R$, typically the \textsc{sat} solution clusters will still be locally inhomogeneous, because $\nu_v$ will fluctuate slightly within each $B_R(v)$ class. It will still be the case that 
$\bm{\Omega}\ll\E(\bm{\Omega}\,|\,\DD_R)$ with high probability.

In prior works that have used the second moment method to lower bound the satisfiability threshold in random $\ksat$, the proof strategies follow the same basic conceptual outline: start from $X=\SOL(\GG)$ or $X=\CLUSTERS(\GG)$, and fix a radius $R$. Then devise some $X_R\subseteq X$ which is locally homogeneous, and perform the second moment method on $|X_R|$ given $\DD_R$ (near $\DD_{\star,R}$). The choices of $R$ and $X$ have been, essentially, as follows:
	\[
	\begin{array}{c|cccc}
	& R & X & X_R & \textup{resulting lower bound on $\asat$}\\
	\hline
	\textup{\cite{1182003}}
		& 0 & \SOL(\GG) & 
			\textup{$\nae$ assignments}
			& 2^{k-1}\log2 - O(1) \\
	\textup{\cite{MR2121043}}
		& 0 & \SOL(\GG)
			& \textup{balanced \textsc{sat} assignments}
			& 2^k\log2-O(k)\\
	\textup{\cite{Coja-Oghlan:2013:GAK:2488608.2488698}}
		& 1/2 & \SOL(\GG)
			& \textup{``judicious'' \textsc{sat} assignments}
			&2^k \log 2 -\tfrac32 \log 2 + \ep_k
			\\
	\textup{\cite{MR3436404}}
		& 1/2 & \CLUSTERS(\GG)
			& \textup{``judicious'' \textsc{sat} clusters}
			&2^k\log 2-\f12(1+ \log 2) - \ep_k.
	\end{array}\]
We use a similar notion of ``judicious'' clusters, and defer the exact definition to Section~\ref{s:preprocess.defns} (Definition~\ref{d:judicious});
the sole purpose of the condition is to enforce local homogeneity.

Of course, the above approach can only succeed when $X_R\ne\emptyset$ (with high probability). However, the set $X$ is in reality \bemph{not} locally homogeneous given $\DD_R$ for any fixed $R$, so we expect
	\[
	\f{|X|}{|X_R|} \le \f1{\exp(n\ep_R)}
	\]
for $\ep_R$ positive but vanishing in the limit $R\to\infty$. This suggests a regime $(\alpha-\delta_R,\asat)$ in which $X_R=\emptyset$ with high probability, where $\delta_R$ is positive but vanishing as $R\to\infty$. Thus, to achieve the exact satisfiability threshold under this scheme, it is necessary to take $R\to\infty$. This is precisely the strategy of this paper: we take $X=\CLUSTERS(\GG)$, and perform the second moment computation on a subset of ``judicious'' configurations
$X_R\subseteq X$, conditional on $\DD_R$. In this way we prove $\asat\ge \arsb-o_R(1)$ where $\arsb$ is the predicted threshold of
Proposition~\ref{p:phi}, and the result follows by taking $R\to\infty$.

\subsection{Combinatorial encoding of clusters}
\label{ss:frozen.model.and.coarsening}

The remainder of the current section is dedicated to the combinatorial representation of $X=\CLUSTERS(\GG)$, where we follow \cite{MR3436404}. For further background, we refer to \cite{parisi_whitening,MR2351840,MR2518205} and the references therein.

Recall that a variable $v$ has incident edges $\delta v$, which connect to its neighboring clauses $\pd v$, where we regard both $\delta v$ and $\pd v$ as unordered multisets.
A $\ksat$ solution is given by a configuration
$\uBX\in\set{\plus,\minus}^V$ such that every clause $a\in F$ is satisfied, meaning that the $k$-tuple $(\lit_{av}\BX_v)_{v\in\pd a}$ is not identically $\minus$. We now introduce a new spin $\free\equiv\SSPIN{free}$, and use it to define the combinatorial model of $\ksat$ solution clusters.

\begin{dfn}[frozen model] \label{d:frozen.model}
Throughout this paper we take the convention that if $\lit\in\set{\minus,\plus}$ and $x=\free$, then 
$\lit x\equiv\free$.
On a $\ksat$ instance $\GG=(V,F,E)$, a \bemph{frozen configuration} is a vector $\ux\in\set{\plus,\minus,\free}^V$ such that
\begin{enumerate}[(i)]
\item \label{i:frozen.sat.clause} Each clause $a\in F$ is \bemph{satisfied}, meaning
that for at least one $e\in\delta a$ we have $\lit_e x_{v(e)} \in \set{\plus,\free}$;
\item \label{i:frozen.var.forced}
A variable $v\in V$ takes value $x_v\ne\free$
if and only if it is \bemph{forced} to do so,
meaning that for at least one $e\in\delta v$ we have
$\lit_{e'}x_{v(e')} = \minus \lit_e x_{v(e)}$ for all $e'\in\delta a(e)\setminus e$.
\end{enumerate}
(The definition makes sense even in the presence of multi-edges.)
\end{dfn}

\begin{rmk}\label{r:coarsen}
Place a graph structure on the set 
$\SOL(\GG)\equiv\GG^{-1}(1)$ of satisfying assignments by putting an edge between any pair of assignments at Hamming distance one, and define a \bemph{cluster} of solutions to be a (maximal) connected component of the graph $\GG^{-1}(1)$. As has been explained in the literature (see e.g.\ \cite{parisi_whitening,MR2351840,MR2518205}), frozen configurations encode clusters in the sense that there is a natural mapping
	\beq\label{e:coarsen}
	\COARSEN:
	\GG^{-1}(1)
	\longrightarrow
	\bigg\{
	\textup{frozen configurations on $\GG$}
	\bigg\}\,.
	\eeq
If $\GG$ has no multi-edges, then each cluster $\cluster\subseteq \GG^{-1}(1)$ maps to a single frozen configuration $\COARSEN(\cluster)$. To define the map, given a configuration $\ux\in\set{\plus,\minus,\free}^V$
that does not violate any clauses
(in the sense of Definition~\ref{d:frozen.model}\eqref{i:frozen.sat.clause}),
 let us say that a variable $v\in V$ is \bemph{blocked} with respect to $\ux$ if there is some $e\in\delta v$ such that $\lit_{e'}x_{v(e')} = \minus \lit_e x_{v(e)}$ for all $e'\in\delta a(e)\setminus e$ (cf.\ Definition~\ref{d:frozen.model}\eqref{i:frozen.var.forced}). Define $\CO(\ux) = \vec{y}$ where
 	\[
	y_v=\left\{\hspace{-3pt}
	\begin{array}{cl}
	x_v & \textup{if $v$ is blocked with respect to $\ux$;}\\
	\free & \textup{otherwise.}
	\end{array}
	\right.
	\]
The map of \eqref{e:coarsen} is defined by iterating $\CO$ until the configuration stabilizes, i.e., $\COARSEN\equiv\CO^\infty$. Note that termination happens in finite time, since if $\CO(\ux)\ne\ux$ then $\CO(\ux)$ has strictly more $\free$ variables than $\ux$. If $\GG$ has no multi-edges, then the value of $\CO(\ux)$ at any variable $v$ is a function of the neighboring values $x_u$, $u\in N(v)$. Consequently, if $\ux$ and $\ux'$ differ in a single coordinate, then $\CO(\ux)=\CO(\ux')$. Iterating this gives $\COARSEN(\ux)=\COARSEN(\ux')$ as long as $\ux,\ux'$ lie in the same cluster. This shows that the map \eqref{e:coarsen} takes each cluster $\cluster$ to a single frozen configuration.
\end{rmk}

\begin{rmk}\label{r:cube}
Another mapping one might consider is
$\CUBE: \GG^{-1}(1) \to \set{\plus,\minus,\free}^V$,
where $\CUBE(\ux)=\vec{y}$ is defined by
	\[y_v=\left\{\hspace{-3pt}
	\begin{array}{cl}
	x_v & \textup{if $v$ only takes value $x_v$
		in the cluster containing $\ux$;}\\
	\free & \textup{otherwise.}
	\end{array}\right.\]
This also has the property that each cluster is mapped to a single point in $\set{\plus,\minus,\free}^V$, which in this case simply encodes the minimal Hamming subcube containing that cluster. In many situations the maps $\COARSEN$ and $\CUBE$ are identical, but one can construct cases where they differ; see Figure~\ref{f:coarsening.example}. The frozen model is preferable precisely because it is defined only by \bemph{local} constraints.
\end{rmk}

\begin{figure}[h!]
\centering
\includegraphics{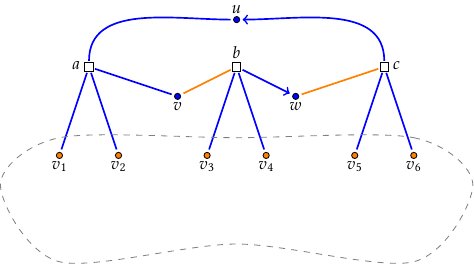}
\caption{An example where the maps $\COARSEN$ and $\CUBE$ differ (Remarks~\ref{r:coarsen}~and~\ref{r:cube}). First ignore the portion of the figure above the variables $v_i$ ($1\le i\le 6$): the dashed gray curve represents a $\ksat$ instance $\GG$, such that $\SOL(\GG)$ has a cluster $\gamma$ in which the variables $v_i$ ($1\le i\le 6$) always take value $x_{v_i}=\minus$. Suppose that in $\GG$ we have
$\COARSEN(\ux)=\CUBE(\ux)=\vec{z}$. Given $\GG$, form a new instance $\GG'\supset\GG$ by adding the rest of the figure, so $\GG'\setminus\GG$ contains the variables $u,v,w$ as well as the clauses $a,b,c$. Each edge $e\in\GG'\setminus\GG$ is colored blue if $\lit_e=\plus$, orange if $\lit_e=\minus$. For $\ux\in\gamma$, we consider how to extend $\ux$ to a satisfying assignment $\uy$ of $\GG'$. If $y_u=\minus$, then clause $a$ forces $y_v=\plus$, and then clause $b$ forces $y_w=\plus$, and then clause $c$ is violated. Consequently we must have $y_u=\plus$, which ensures that clauses $a$ and $c$ are satisfied. For clause $b$ to be satisfied, we must have $(y_v,y_w)\ne(\plus,\minus)$. Thus a valid extension of $\ux\in\gamma$ is given by $\uy=(\ux,\plus,\plus,\plus)$ where $(\plus,\plus,\plus)$ indicates the values on $(u,v,w)$. From the above discussion, in $\GG'$ we have
$\CUBE(\uy)=(\vec{z},\plus,\free,\free)$. On the other hand, the arrows in the figure indicate that in the initial configuration $\uy=(\ux,\plus,\plus,\plus)$, variable $w$ is blocked by cluase $b$, and variable $u$ is blocked by clause $c$. Thus variable $v$ is unblocked so the $\COARSEN$ map first changes $y_v$ to $\free$. This makes $w$ unblocked, so next $y_w$ is changed to $\free$, which makes $u$ unblocked. As a result $\COARSEN(\uy)=(\vec{z},\free,\free,\free)\ne\CUBE(\uy)$. We emphasize that this example required the introduction of a short cycle in $\GG'\setminus\GG$, and so we expect such occurrences to be rare in the random $\ksat$ model.}
\label{f:coarsening.example}
\end{figure}

\subsection{Warning propagation and color model}
\label{ss:wp.and.color.models}

We now introduce two more combinatorial models, which are both equivalent to the frozen model, but will be important for analytical purposes. The first is the ``warning propagation'' (\textsc{wp}) model which has appeared widely in the physics literature (see \cite{MPZ_Science, MR2155706, MR2213115, MR2351840, MR2518205}). In this model, a solution cluster is represented by a ``warning configuration'' 
$\vec{\msg}\equiv(\msg_e)_{e\in E}$,
where each
$\msg_e\equiv(\dmp_e,\hmp_e)$ 
represents a pair of ``warnings'' sent across $e$ in either direction:
	\[\begin{array}{l}
	\dmp_e \text{ represents the
		warning across $e$ from $v(e)$ to $a(e)$;}\\
	\hmp_e \text{ represents the
		warning across $e$ from $a(e)$ to $v(e)$.}
	\end{array}
	\]
Each warning concerns the evaluation of $\lit_e x_{v(e)}$ within the cluster; the warning from each endpoint of $e$ represents the state of $\lit_e x_{v(e)}$ ``in absence of'' the opposite endpoint. The possible warnings are $\plus,\minus,\free$, where $\PM$ indicates a warning that $\lit_e x_{v(e)}$ must be $\PM$ (within the cluster), while $\free$ indicates no warning. In the warning propagation model, the clause $a(e)$ can only force the variable $v(e)$ to agree with $\lit_e$, so all clause-to-variable warnings $\hmp_e$ will be either $\plus$ or $\free$. The formal definition is as follows:

\begin{dfn}[warning propagation model]
\label{d:wp} For integers $K\ge1$ and $D\ge0$, define the mappings
	\begin{align*}
	\hat{W}(\dmp_1,\ldots,\dmp_K)
	&= \left\{\begin{array}{rl}
	\plus &\textup{if }
	 \dmp_i=\minus
		\textup{ for all }1\le i\le K\,,\\
	\free & \textup{otherwise;}
	\end{array}\right. \\
	\dot{W}(\hmp_1,\ldots,\hmp_D)
	&=\left\{\begin{array}{rl}
	\plus & \textup{if }
		\plus \in \set{\hmp_1,\ldots,\hmp_D}
			\subseteq \set{\plus,\free}\,,\\
	\minus & \textup{if }
		\plus \in \set{\hmp_1,\ldots,\hmp_D}
			\subseteq \set{\minus,\free}\,,\\
	\free &\textup{if }
		\set{\hmp_1,\ldots,\hmp_D}
		\subseteq\set{\free}\,,\\
	\emptyset & \textup{otherwise}.
	\end{array}\right.
	\end{align*}
In the case $D=0$, the input to the function $\dot{W}$ is empty, and the output is $\free$. On a $\ksat$ instance $\GG\equiv(V,F,E)$, a \bemph{warning configuration} is a tuple $\vec{\msg}$ of spins $\msg_e \equiv(\dmp_e,\hmp_e) \in\set{\plus,\minus,\free}\times\set{\plus,\free}$, indexed by edges $e\in E$, such that the \bemph{warning propagation equations} are satisfied:
	\begin{align}\nonumber
	\hmp_e
	=\hWP_e[\dmp]
	&\equiv\hat{W}_{k-1}
	\bigg(
	\dmp_{e'} : e' \in \delta a(e)\setminus e
	\bigg)\,,\\
	\dmp_e
	=\dWP_e[\hmp]
	&\equiv \dot{W}
	\bigg( \lit_{e'} \lit_e\hmp_{e'} 
	: e' \in \delta v(e)\setminus e \bigg)\,.
	\label{e:dot.WP}
	\end{align}
for all edges $e\in E$. (Note that for $\vec{\msg}$ to be a valid warning configuration, we require each $\dmp_e$ to be an element of $\set{\plus,\minus,\free}$, which means $\dWP_e$ must not output $\emptyset$ for any $e\in E$. We remark also that taking $\dmp_e=\hmp_e=\free$ for all $e\in E$
gives rise always to a valid warning configuration.) \end{dfn}

\begin{rmk}\label{r:bij.wp.frozen} Warning configurations are in bijective correspondence with frozen configurations of $\GG$. If $\vec{\msg}$ is a valid warning configuration of $\GG$, we can obtain a valid frozen configuration $\ux$ of $\GG$ simply by taking
	\[
	x_v = \dot{W}\bigg( \lit_e \hmp_e : e\in\delta v\bigg)
	\]
for all $v\in V$. Conversely, if $\ux$ is a valid frozen configuration of $\GG$, we can obtain a valid warning configuration $\vec{\msg}$ of $\GG$ by to the following procedure:
\begin{enumerate}[(i)]
\item If $e\in\delta v$ with
$x_v=\free$, then we must have $\msg_e=(\free,\free)$.

\item If $e\in\delta v$ with $x_v\in\set{\plus,\minus}$ 
and $\lit_e x_v=\minus$, then we must have $\msg_e=(\minus,\free)$.

\item If $e\in\delta v$ with $x_v\in\set{\plus,\minus}$ 
and $\lit_e x_v=\plus$, then
	\[
	\hmp_e
	=\left\{\begin{array}{rl}
	\plus & \textup{if 
		$\msg_{e'}=(\minus,\free)$
		for all $e'\in\delta a(e)\setminus e$}\\
	\free & \textup{otherwise.}
	\end{array}\right.
	\]

\item The above determines of $\hmp_e$ for all $e$. We can then determine $\dmp_e$ for all $e$ by applying the \textsc{wp} rules \eqref{e:dot.WP}.
\end{enumerate}
We leave the reader to verify that these mappings are inverses of one another.
\end{rmk}

In fact, we will rarely use the full warning propagation model, and instead will work primarily with the following especially concise simplification of \textsc{wp}, introduced by \cite{MR3436404}. 
Note that, under the rules of Definition~\ref{d:wp}, the pair $\msg_e\equiv(\dmp_e,\hmp_e)$ can take any value in $\set{\plus,\minus,\free}\times\set{\plus,\free}$ except for $(\minus,\plus)$, which represents a pair of conflicting warnings that will invalidate the entire configuration (since in this case $\dWP_{e'}$ will output $\emptyset$ for some $e'\in\delta v(e)\setminus e$.

\begin{dfn}[color model] \label{d:color.model}
Let $\PROJ$ be the mapping which sends warnings $\msg\equiv(\dmp,\hmp)$ to \bemph{colors} $\sigma\in\set{\red,\yel,\grn,\blu}$, according to the following rules:
	\[
	\PROJ(\dmp,\hmp)
	=\left\{
	\begin{array}{rll}
	\red\equiv\SPIN{red}
	& \textup{if $\hmp=\plus$
	and $\dmp\in\set{\plus,\minus}$}
	& \textup{(clause forces this edge to be satisfied),}\\
	\yel\equiv\SSPIN{yellow}
	& \textup{if 
	$(\dmp,\hmp)=(\minus,\free)$}
	&\textup{(variable forced to negate edge;
		clause not forcing on edge),}\\
	\grn\equiv\SSPIN{green}
	& \textup{if 
	$(\dmp,\hmp)=(\free,\free)$}
	&\textup{(no forcing warnings across this edge)},\\
	\blu\equiv\SSPIN{blue}
	&\textup{if 
	$(\dmp,\hmp)=(\free,\plus)$}
	&\textup{(variable forced to affirm edge;
		clause not forcing on edge).}
	\end{array}
	\right.
	\]
A valid \bemph{coloring} of $\GG=(V,F,E)$ is any
configuration $\usi\in\set{\RYGB}^E$ that can be obtained by taking a valid warning configuration $\vec{\msg}$
and applying $\PROJ$ on each edge.
\end{dfn}

\begin{rmk}\label{r:bij.color.frozen}
Colorings are in bijective correspondence 
with frozen configurations of $\GG$. To see this,
let $\usi_{\delta v}$ denote the colors
on the incident edges of variable $v$, 
recall the notation \eqref{e:var.signed.half.edges}, and define
	\beq\label{e:EVAL.of.color}
	x_v = \ev_v(\usi_{\delta v})
	\equiv\left\{\hspace{-5pt}\begin{array}{rll}
	\plus & \text{if
		$\usi_{\delta v(\minus)}$
		is all $\yel$, and
		$\usi_{\delta v(\plus)}$
		is all $\set{\red,\blu}$
		with at least one $\red$}
		\\
	\minus & \text{if $\usi_{\delta v(\plus)}$
		is all $\yel$, and
		$\usi_{\delta v(\minus)}$
		is all $\set{\red,\blu}$
		with at least one $\red$}
		\\
	\free & \text{if 
		$\usi_{\delta v}$ is all $\grn$}
		\\
	\emptyset & \text{otherwise.}
	\end{array}\right.\eeq
If $\usi$ is a valid coloring of $\GG$, then taking $x_v=\ev_v(\usi_{\delta v})$ for all $v\in V$ defines a valid frozen configuration $\ux$ of $\GG$. By Remark~\ref{r:bij.wp.frozen}, $\ux$ corresponds to a unique warning configuration $\vec{\msg}$. We leave the reader to verify that $\PROJ$ maps $\vec{\msg}$ back to the starting $\usi$, which completes the correspondence.
\end{rmk}

It is essential to us that all the combinatorial models 
that we consider (Definitions~\ref{d:frozen.model}, \ref{d:wp}, and \ref{d:color.model}) are defined only by \bemph{local} constraints. This implies that they are all \bemph{factor models} --- here, it simply means that the counting measure on valid configurations can be expressed as a product of \bemph{local factors}, where each factor is an indicator function that checks one of the local constraints. For factor models in much greater generality, there is a rich physics formalism (see e.g. \cite[Ch.~9]{MR2518205}), as well as a natural way to turn the moment calculation into an analytic optimization problem
problem (see \S\ref{ss:first.moment.preliminaries}).

Let $\GG=(V,F,E)$ be a random $\ksat$ instance. We will give the explicit factors for the \textsc{wp} and color models. The counting measure on warning configurations
of $\GG$ is given simply by
	\[
	\mathbf{1}\Bigg\{\hspace{-3pt}
	\begin{array}{c}
	\textup{$\vec{\msg}$
	is a valid warning}\\
	\textup{configuration on $\GG$}
	\end{array}\hspace{-3pt}
	\Bigg\}
	= \prod_{v\in V}
		\varphi_v(\vec{\msg}_{\delta v})
	\prod_{a\in F}\hat{\varphi}_a(\vec{\msg}_{\delta a})
	\]
where each factor $\varphi_x$ simply checks that the warnings leaving $x$ are indeed obtained by applying the appropriate warning propagation maps on the incoming warnings:
	\begin{align*}
	\varphi_v(\vec{\msg}_{\delta v})
	&\equiv
	\prod_{e\in\delta v}
	\Ind{\dmp_e = \dWP_e[\hmp] }\,,\\
	\hat{\varphi}_a(\vec{\msg}_{\delta a})
	&\equiv\prod_{e\in\delta a}
	\Ind{\hmp_e = \hWP_e[\dmp]}\,.
	\end{align*}
We will abuse notation slightly and also use $\varphi,\hat{\varphi}$ for the factors of the coloring model. The counting measure on valid colorings is given by
	\beq\label{e:color.factor.model}
	\mathbf{1}\Bigg\{\hspace{-3pt}
	\begin{array}{c}
	\textup{$\usi$
	is a valid}\\
	\textup{coloring of $\GG$}
	\end{array}\hspace{-3pt}\Bigg\}
	=
	\prod_{v\in V}\varphi_v(\usi_{\delta v})
	\prod_{a\in F}
		\hat{\varphi}_a(\usi_{\delta a})\,,
	\eeq
where the variable factor $\varphi_v(\usi_{\delta v})$
is simply the indicator that
$\usi_{\delta v}$
can be obtained by applying $\PROJ$
to a configuration
$\vec{\msg}_{\delta v}$ for which $\varphi_v(\vec{\msg}_{\delta v})=1$; and the clause factor $\hat{\varphi}_a$
is analogously defined. The color model factors can be described much more explicitly, as follows:
the clause factor is
	\beq\label{e:indicator.of.valid.clause.coloring}
	\hat{\varphi}_a(\usi_{\delta a})
	=
	\left\{\hspace{-5pt}
	\begin{array}{rll}
	1 & \text{if $\usi_{\delta a}$ has exactly one spin $\red$, all other spins $\yel$}
		& \text{(forcing clause);}\\
	1 & \text{if $\usi_{\delta a}$ has
	no spins $\red$ and at least two spins $\set{\grn,\blu}$}
		& \text{(non-forcing clause);}\\
	0 & \text{otherwise.}
	\end{array}\right.
	\eeq
(As far as clauses are concerned, the colors $\blu$ and $\grn$ are interchangeable.)
The variable factor is given by
	\beq\label{e:color.model.variable.factor}
	\varphi_v(\usi_{\delta v})
	= \mathbf{1}\bigg\{
	\ev_v(\usi_{\delta v}) \ne\emptyset 
	\bigg\}\,,\eeq
for $\ev_v$ defined by \eqref{e:EVAL.of.color}.
Note that the clause factor does not depend on the incident edge signs, but the variable factor does, since the edge signs enter into the definition of $\ev_v$.

\begin{rmk}\label{r:single.copy.pair}
We will sometimes refer to the color model above as the \bemph{single-copy} color model, to distinguish it from the \bemph{pair} color model which
appears in the second moment: the latter is supported on pairs $(\usi^1,\usi^2)$ with each $\usi^j$ a valid coloring of the same graph $\GG$. If we extend $\varphi_v,\hat{\varphi}_a$ to pair inputs by setting
	\begin{align*}
	\varphi_v(\usi^1_{\delta v},
		\usi^2_{\delta v})
	&\equiv
	\prod_{j=1,2}\varphi_v(\usi^j_{\delta v})\,,\\
	\hat{\varphi}_a(\usi^1_{\delta a},
		\usi^2_{\delta a})
	&\equiv
	\prod_{j=1,2}\hat{\varphi}_a(\usi^j_{\delta a})\,,
	\end{align*}
then the counting measure for the pair model is also expressed by \eqref{e:color.factor.model}, provided that for each vertex $x\in V \cup F$ we interpret $\usi_{\delta x}$ as $\equiv(\usi^1_{\delta x},\usi^2_{\delta x})$.
\end{rmk}

To summarize what we have discussed 
in \S\ref{ss:frozen.model.and.coarsening}--\ref{ss:wp.and.color.models}:
on a given $\ksat$ problem instance $\GG=(V,F,E)$,
assuming there are no multi-edges,
there is a map $\COARSEN$ which sends solution clusters to frozen configurations (Remark~\ref{r:coarsen}). We then have bijections
(Remarks~\ref{r:bij.wp.frozen}~and~\ref{r:bij.color.frozen})
	\beq\label{e:aux.model.bijections}
	\left\{\hspace{-4pt}
	\begin{array}{c}
	\textup{frozen configurations}\\
	\textup{$\ux \in\set{\plus,\minus,\free}^V$}
	\end{array}
	\hspace{-4pt}\right\}
	\longleftrightarrow
	\left\{\hspace{-4pt}
	\begin{array}{c}
	\textup{warning configurations}\\
	\textup{$\vec{\msg}\in
		\set{
		\plus\plus,
		\plus\free,
		\minus\free,
		\free\plus,
		\free\free
		}^E$}
	\end{array}
	\hspace{-4pt}\right\}
	\longleftrightarrow
	\left\{\hspace{-4pt}
	\begin{array}{c}
	\textup{valid colorings}\\
	\textup{$\usi\in\set{\RYGB}^E$}
	\end{array}
	\hspace{-4pt}\right\}.
	\eeq
The models are all defined by local constraints, and so can be written as factor models. For the color model the factors can be written in an especially explicit way, given above by \eqref{e:EVAL.of.color}, \eqref{e:indicator.of.valid.clause.coloring}, and \eqref{e:color.model.variable.factor}. In the remainder of this section we introduce \bemph{tree recursions} for these models.

\subsection{Tree recursions for warnings}
\label{ss:wp.recursions}

To motivate what comes next, we note that the central analytical challenge of this paper is to prove a lower bound for colorings (Definition~\ref{d:color.model}) on the random $\ksat$ graph $\GG$ (Definition~\ref{d:formal.random.ksat}). In view of the bijection \eqref{e:aux.model.bijections}, this translates to a lower bound on frozen configurations (Definition~\ref{d:frozen.model}), which we then show can be ``completed'' to satisfying assignments of $\GG$. This short synopsis hides many technicalities which will appear later. For now, however, we will focus on the fundamental issue of local homogeneity, and consider a basic question: if $\gamma$ is a uniformly random solution cluster and $\ux$ is the corresponding frozen configuration, how should the marginal law of $x_v$ depend on the local neighborhood structure around $v$? We do not directly answer this question, but in the rest of this section we describe the physics prediction in mathematical terms. In later sections we use the prediction to obtain the rigorous result.

Let $\GG=(V,F,E)$ be any \textsc{sat} instance. Recall the notation \eqref{e:var.signed.half.edges}.
If $v$ does not participate in any multi-edges,
then for each clause $a\in\pd v$ we shall denote
	{\setlength{\jot}{0pt}
	\begin{alignat}{2}\nonumber
	\pd v(\plus a)
	&\equiv\set{
		b\in\pd v\setminus a
		: \lit_{bv} = \plus \lit_{av} }\,,\quad
	&\delta v(\plus a)
		&\equiv\set{(bv) : b\in \pd v(\plus a)}\,,\\
	\pd v(\minus a)
	&\equiv\set{ b\in\pd v\setminus a
		: \lit_{bv} = \minus \lit_{av} }\,,
	&\delta v(\minus a)
		&\equiv\set{(bv) : b\in \pd v(\minus a)}\,.
	\label{e:incident.edges.of.same.sign}
	\end{alignat}}%
For any vertex $x\in V\cup F$ we can consider its depth-$\ell$ neighborhood $B_\ell(x)$, which we regard as being rooted at $x$. If $x\in V$ we assume that $\ell$ is a nonnegative integer; if $x\in F$ we assume that $\ell-1/2$ is a nonnegative integer. If $\ell$ is any constant and $\GG$ is an instance of \bemph{random} $\ksat$, then $B_\ell(x)$ will be acyclic for most vertices of $\GG$, with high probability. Since this is the predominant case, we will assume it for the remainder of the section. If $B_\ell(x)$ is acylic then it is a \bemph{bipartite factor tree}, by which we mean a bipartite factor graph that is also a tree. We will write $B_\ell(x)\equiv T \equiv (V_T,F_T,E_T)$. 

Our next goal is to define a frozen model on $T=B_\ell(x)$ that will approximate, in some sense, what we expect to see in the frozen model in the full random graph $\GG$. To do this, we will impose certain \bemph{boundary conditions} on $T$, which we now explain. Let $\Leaves T$ denote the leaf vertices of $T$; they are all variables by the conditions on $\ell$. The set $\Leaves T$ contains (but may be strictly larger than) the set $\pd_\ell x$ of variables at distance exactly $\ell$ from $x$. We now take advantage of the correspondence \eqref{e:aux.model.bijections} and think in terms of warning configurations.
Let $\vec{\dmp}_\pd$ denote a vector of boundary input warnings $\dmp_{u a(u)}$, where $u$ runs over $\Leaves T$ and $a(u)$ denotes the unique clause neighboring $u$. Any $\vec{\dmp}_\pd$ has at most one completion to a valid warning configuration $\vec{\msg}$ on $T$. One way to see this is to apply the \textsc{wp} maps \eqref{e:dot.WP} of Definition~\ref{d:wp} recursively, started from the boundary inputs $\vec{\dmp}_\pd$: either at some point $\dWP$ outputs $\emptyset$ and there is no valid completion of $\vec{\dmp}_\pd$, or the process eventually terminates at the unique valid completion $\vec{\msg}$ of $\vec{\dmp}_\pd$. For example, if $\lit_e=\plus$ for all edges $e$ in $T$, and all entries of 
$\vec{\dmp}_\pd$ are in $\set{\plus,\free}$, then 
it is easy to see that $\vec{\dmp}_\pd$ can be completed to a valid warning configuration $\vec{\msg}$ on $T$, with $\dmp_e=\hmp_e=\free$ on every internal edge $e$ of $T$. 

\begin{dfn}\label{d:frozen.model.bdy.conditions} On the tree $T=B_\ell(x)$, let $\vec{\dmp}_\pd$ be a tuple of random boundary input warnings, where each $\dmp_{u a(u)}$ is sampled independently from $\textup{unif}(\set{\plus,\minus})$ if $u\in\pd_\ell x$, and $\dmp_{u a(u)}=\free$ for $u\in\Leaves T\setminus \pd_\ell x$. Explicitly,
	\beq\label{e:frozen.model.bdy.conditions}
	\mu_{\pd}(\vec{\dmp}_\pd)
	= \bigg\{ \prod_{u\in\pd_\ell x}
	\f{\Ind{ \dot{\msg}_{u a(u)}
		\in\set{\plus,\minus} }}{2}
		\bigg\}\bigg\{
	\prod_{w\in \Leaves T \setminus \pd_\ell x}
	\Ind{ \dot{\msg}_{w a(w)} = \free }\bigg\}\,.
	\eeq
Let $\nu_T$ be the law of the completion $\vec{\msg}$ of $\vec{\dmp}_\pd$, conditioned on the event that the completion exists:
	\[
	\nu_T(\vec{\msg})
	\cong
	\sum_{\vec{\dmp}_\pd}
	\mathbf{1}\bigg\{
	\textup{$\vec{\msg}$ is the completion of
		$\vec{\dmp}_\pd$ on $T$}
	\bigg\}
	\mu_{\pd}(\vec{\dmp}_\pd)
	\]
where $\cong$ indicates proportionality up to the normalization that makes $\nu_T$ a probability measure. The measure $\nu_T$ is well-defined. Indeed, in the special case that $\lit_e=\plus$ for all edges $e$ of $T$, we noted above that $\vec{\dmp}_\pd$ has a valid completion as long as all its entries are in $\set{\plus,\free}$, which guarantees that $\nu_T$ is well-defined. Similarly, for general $T$, there is always at least one choice of $\vec{\dmp}_\pd$ 
that has a valid completion on $T$
and has $\mu_\pd(\vec{\dmp}_\pd)>0$, so $\nu_T$ is well-defined in general. By passing through the bijection \eqref{e:aux.model.bijections}, $\nu_T$ induces a probability measure on valid frozen configurations of $T$, which we term the \bemph{frozen model on $T$ with rigid boundary}.\footnote{We chose the ``rigid boundary'' terminology because it is reasonably succinct, but note that it is not fully descriptive since the variables in $\Leaves T \setminus \pd_r v$ are free rather than rigid.}
\end{dfn}

This tree frozen model is characterized by a set of recursions, as follows. For any variable-clause edge $(yz)$ in $T$ (where either $y$ or $z$ is the variable), let $T_{yz}$ be the component of $T\setminus z$ containing $y$ --- including the edge $(yz)$, but not including $z$ itself. If $y$ is a variable then we call $T_{yz}$ a \bemph{variable-to-clause tree}.
If $y$ is a clause then we call $T_{yz}$ a \bemph{clause-to-variable tree}. In either case, we write $\Leaves T_{yz}$ as shorthand for
$T_{yz}\cap \Leaves T$. Given inputs $\dmp_{va}$ for all $v\in \Leaves T_{yz}$, we can apply the \textsc{wp} maps (\eqref{e:dot.WP} from Definition~\ref{d:wp}) recursively started from $\Leaves T_{yz}$. As long as $\dWP$ never outputs $\emptyset$, this produces all the warnings on $T_{yz}$ in the direction of $z$, which we term a \bemph{completion on $T_{yz}$}, and denote by $\vec{\msg}_{T,yz}$. Let $\mu_{T,yz}$ denote the marginal of $\mu_{\Leaves T}$ (as defined by \eqref{e:frozen.model.bdy.conditions}) on $\Leaves T_{yz}$, and let
	\beq\label{e:def.FF.of.tree}
	\nu_{T,yz}(\vec{\msg}_{T,yz})
	\cong
	\sum_{ \vec{\dmp}_{\Leaves T_{yz}} }
	\mathbf{1}\bigg\{
	\textup{$\vec{\msg}_{T,yz}$
	is the completion of $\vec{\dmp}_{\Leaves T_{yz}}$
	 on $T_{yz}$
	}
	\bigg\}
	\mu_{T,yz}(\vec{\dmp}_{\Leaves T_{yz}})\,,\eeq
where again $\cong$ refers to the normalization. Let $\FF_{T,yz}$ be the marginal law of the $y\to z$ warning	under $\nu_{T,yz}$. We will also designate some special notation for marginal laws of warnings around the root of the tree, as follows. For $T=B_\ell(x)$ and $y\in\pd x$:
\begin{enumerate}[--]
\item The law $\FF_{T,xy}$ 
 depends only on the subtree $T_{xy}$, so we write $\FF_{T,xy} \equiv \FF_\ell(T_{xy})$, where the subscript $\ell$ indicates that the vertices at distance exactly $\ell$ from $(xy)$ have a special role in \eqref{e:frozen.model.bdy.conditions}.
\item The law $\FF_{T,yx}$ 
depends only on the subtree $T_{yx}$,
so we write $\FF_{T,yx} \equiv \FF_{\ell-1/2}(T_{yx})$, 
where the subscript $\ell-1/2$ indicates that the vertices at distance exactly $\ell-1/2$ from $(xy)$ have a special role in \eqref{e:frozen.model.bdy.conditions}.
\end{enumerate} 
Finally, if $U_{xy}$ is any tree that agrees with $T_{xy}$ up to distance $\ell$ around $(xy)$, then we define $\FF_\ell(U_{xy})\equiv\FF_\ell(T_{xy})$.
This notation will appear again in the following sections.

For $T=B_r(v)$, for each edge $(au)\in E_T$ where $a$ is the clause and $u$ is the variable, we will abbreviate $\bhu_{au} \equiv \FF_{T,au}$ and $\bmeta_{ua}\equiv\FF_{T,ua}$ for the marginal laws of $\hmp_{au}$ and $\dmp_{ua}$ induced by \eqref{e:def.FF.of.tree}. These are both probability measures on $\set{\plus,\minus,\free}$. We also denote the scalar values
$\eta_{ua}\equiv\bmeta_{ua}(\minus)$
and $\hat{u}_{au}\equiv\bhu_{au}(\plus)$.
If a variable $u$ lies at depth $r$ in $T$ (i.e., $u\in\pd_r v$), then it neighbors exactly one clause $a=a(u)\in T$, and the definition gives
$\bmeta_{ua}=\textup{unif}(\set{\plus,\minus})$
and consequently $\eta_{ua}=1/2$.
If instead the variable $u$ is a leaf at distance less than $r$ from $v$ (i.e., $u\in\Leaves T\setminus\pd_r v$),
then the definition makes
$\bmeta_{ua}$ fully supported on $\set{\free}$,
and consequently $\eta_{ua}=0$.
This defines $\eta$ on every leaf edge of $T$. 
The value of $\eta$ on every other edge can be computed recursively:
\begin{enumerate}[--]
\item For an edge $(aw)\in E_T$ where $\eta_{ua}$ was already determined for all $u\in\pd a\setminus w$,
	\beq\label{e:first.defn.of.bhu}
	\bhu\equiv
	\bigg(\bhu_{aw}(\plus)\,,\bhu_{aw}(\free)\bigg)
	=
	\bigg(
	\prod_{u\in\pd a\setminus w} \eta_{ua}\,,
	1-\prod_{u\in\pd a\setminus w} \eta_{ua}
	\bigg)\,.
	\eeq
In particular this determines the value of $\hat{u}_{aw}\equiv \bhu_{aw}(\plus)$.

\item For an edge $(aw)\in E_T$ where $\hat{u}_{bw}$
was already determined for all $b\in\pd w\setminus a$, let
	\[
	\Pi_{wa}^\plus
	= \prod_{b\in\pd w(\plus a)} 
		\bigg(1- \hat{u}_{bw}\bigg)\,,\quad
	\Pi_{wa}^\minus
	= \prod_{b\in\pd w(\minus a)} 
		\bigg(
		1- \hat{u}_{bw}\bigg)\,,\]
where $\pd w(\plus a)$ and $\pd w(\minus a)$
are defined by \eqref{e:incident.edges.of.same.sign}. Then
	\beq\label{e:intro.defn.bmeta}
	\bigg(\bmeta_{wa}(\plus)\,,\bmeta_{wa}(\minus)\,,
	\bmeta_{wa}(\free)\bigg)
	=\f{ (\Pi^\minus_{wa}(1-\Pi^\plus_{wa})\,,
		\Pi^\plus_{wa}(1-\Pi^\minus_{wa})\,,
		\Pi^\plus_{wa}\Pi^\minus_{wa}
		)}{ \Pi_{wa}^\plus+\Pi_{wa}^\minus
			-\Pi_{wa}^\plus\Pi_{wa}^\minus}\,.
	\eeq
In particular this determines the value of
$\eta_{wa}\equiv \bmeta_{wa}(\minus)$.
\end{enumerate}
For any non-leaf variable $w\in T\setminus \Leaves T$, combining the steps above
shows that $\eta_{wa}$ can be expressed in terms of the $\eta_{ub}$ (for $b\in\pd w\setminus a$ and $u \in \pd b\setminus w$) by the recursive relation
	\beq\label{e:second.introduction.of.Rec}
	\eta_{wa} = R_{wa}(\minus)
	= \f{ \Pi^\plus_{wa}(1-\Pi^\minus_{wa}) }
	{ \Pi_{wa}^\plus+\Pi_{wa}^\minus
			-\Pi_{wa}^\plus\Pi_{wa}^\minus }\,.
	\eeq
(Note the clear similarity between
\eqref{e:second.introduction.of.Rec}
 and the ``survey propagation'' equations introduced in \eqref{e:intro.dist.recurs}.) The probability measures $\bmeta_{wa}$ and $\bm{\hat{u}}_{aw}$ can also be computed from this recursion, as indicated by \eqref{e:first.defn.of.bhu}~and~\eqref{e:intro.defn.bmeta}. 
Finally, note that since we start from boundary input $\eta_{ua(u)}\in[0,1)$ for all $u\in\Leaves T$, we have for all other $(aw)\in E_T$ that $\Pi^\PM_{wa}\in(0,1]$, and so $\eta_{wa}\in[0,1)$ for all $(aw)\in E_T$. The empty product is understood to be one, so if $w\notin\Leaves T$ with $\pd w(\minus a)=\emptyset$ then $\Pi^\minus_{wa}=1$, implying $\eta_{wa}=0$.

\subsection{Weighted models and belief propagation}
\label{ss:bp} 

We will show at the end of this section
that the frozen model recursions of \S\ref{ss:wp.recursions} can be retrieved as a special case of the belief propagation (\textsc{bp}) equations for the color model. In preparation, we will briefly review \textsc{bp} in the slightly generalized setting of \bemph{weighted} color models, which will be used throughout our proof, and which we now introduce. We limit our discussion here to the single-copy color model; the definitions and notations generalize to the pair model (Remark~\ref{r:single.copy.pair}) in the obvious manner.

A weighted color model on $\GG=(V,F,E)$ is defined by multiplying \eqref{e:color.factor.model} with edge weights $\gm_e:\set{\RYGB}\to(0,\infty)$ for all $e\in E$. For our purposes it will often be convenient to consider the weight $\gm_e$ on edge $e=(av)$ as ``belonging'' to the incident clause $a$: that is to say, we replace $\hat{\varphi}_a$ in \eqref{e:color.factor.model} by the weighted clause factor
	\[
	\hat{\varphi}_a(\usi_{\delta a};\Gm_a)
	\equiv
	\hat{\varphi}_a(\usi_{\delta a})
	\prod_{e\in\delta a} \gm_e(\sigma_e)
	\quad\text{where 
	$\Gm_a$
	denotes the tuple 
	$(\gm_e)_{e\in\delta a}$}.
	\]
At other times it is more convenient to consider the edge weight as ``belonging'' to the incident variable $v$. If $v$ is an internal variable of $\GG$ ($|\delta v|\ge2$), it will be natural to parametrize the weights in a slightly different way: namely, we set a weight
$\lm_v$ for the variable's frozen model spin
$x_v=\ev_v(\usi_{\delta v})\in\set{\plus,\minus,\free}$, then set weights $\lm_e$ for the incident edge colors $\sigma_e$, so that
 $\varphi_v$ is replaced by the weighted variable factor
	\beq\label{e:weighted.var.factor}
	\varphi_v(\usi_{\delta v};\Lm_v)
	=
	\varphi_v(\usi_{\delta v})
		\lm_v(x_v)
		\prod_{e\in\delta v}
		\lm_e(\sigma_e)\quad
	\text{where $\Lm_v$ denotes
		$(\lm_v,(\lm_e)_{e\in\delta v})$.}
	\eeq
If $v$ is a leaf variable of $\GG$ ($|\delta v|=1$), we will denote $\Lm_v\equiv\lm_v\equiv\lm_e$ for the weight on the unique edge $e$ incident to $v$.
In this case we have simply $\varphi_v(\usi_{\delta v};\Lm_v)=\lm_e(\sigma_e)$. Inserting the weighted factors into \eqref{e:color.factor.model} defines a weighted measure on valid colorings of $\GG$, which we refer to as the \bemph{weighted color model}:
	\beq\label{e:color.factor.model.WEIGHTED}
	\prod_{v\in V}\varphi_v(\usi_{\delta v};\Lm_v)
	\prod_{a\in F}
		\hat{\varphi}_a(\usi_{\delta a};\Gm_a)\,.
	\eeq
Clearly, scaling any $\Gm_a$ or $\Lm_v$ by a positive constant has no effect other than to scale the entire measure, so we will always anchor the clause weights by fixing the convention $\gm_e(\yel)=1$.
Likewise we anchor the variable weights by fixing
$\lm_v(\plus)=1$ and $\lm_e(\sigma)=1$ for all $\sigma\in\set{\yel,\grn,\blu}$. We will often denote $\lm_e\equiv\lm_e(\red)$.

Let us remark that clause weights can be re-interpreted as variable weights, and vice versa, simply by shifting the factors: for example, a single-copy model consisting of $\Gm$-weighted clauses
surrounded by unweighted variables
can be transformed into a model
with $\Lm$-weighted variables and unweighted clauses by setting
	\begin{align} \nonumber
	\lm_v(\minus) 
	&=
	\prod_{e\in\delta v(\plus)}
		\f{\gm_e(\yel)}{\gm_e(\blu)}
	\prod_{e\in\delta v(\minus)}
		\f{\gm_e(\blu)}{\gm_e(\yel)}\,,\\ \nonumber
	\lm_v(\free) &=
		\prod_{e\in\delta v(\plus)}
		\f{\gm_e(\grn)}{\gm_e(\blu)}
		\prod_{e\in\delta v(\minus)}
		\f{\gm_e(\grn)}{\gm_e(\yel)}\,, \\
	\lm_e(\red)
		&= \f{\gm_e(\red)}{\gm_e(\blu)}
		\textup{ for all $e\in\delta v$}\,,
	\label{e:redistribute.weights}
	\end{align}
recalling that all the other weights 
are fixed at one. 

%

We now briefly review the belief propagation (\textsc{bp}) method in the context of the weighted color model. We shall apply \textsc{bp} only when the underlying bipartite factor graph is a finite tree --- it is well known that \textsc{bp} is exact for this setting, though it is only a heuristic on more general graphs. The reader is referred to \cite{MR2518205} for a detailed introduction to \textsc{bp} for a far broader class of models.

Let $T=(V,F,E)$ be a finite bipartite factor tree. For simplicity we assume that the leaves of $T$ are all variables. Consider a weighted color model \eqref{e:color.factor.model.WEIGHTED} on $T$. Provided this measure has positive mass, we can normalize it to be a probability distribution $\nu$ supported on valid colorings $\usi\equiv\usi_T$ of the tree $T$. The measure $\nu$ is an example of what is more generally termed a \bemph{Gibbs measure}. Belief propagation is a way of computing local marginals of a Gibbs measure $\nu$ on a finite tree $T$: for any subgraph $U\subseteq T$, write $\nu_U$ for the marginal of $\nu$ on $U$. The local marginals $\nu_U$ are efficiently computed in terms of the solution to a system of equations, known as the \textsc{bp} recursions or \textsc{bp} fixed-point equations, as we now describe.

Let $\mathscr{M}$ denote the space of probability measures over $\set{\RYGB}$. The \textsc{bp} recursions are equations defined in terms of the \textsc{bp} \bemph{messages} $\dq,\hq\in\mathscr{M}$ indexed by edges $e=(av)$:
	{\setlength{\jot}{0pt}\begin{align*}
	\dq_e\equiv \dq_{va}
	\equiv \dq_{v\to a}
	&\equiv \textup{variable-to-clause message
		from $v$ to $a$}\\
	&=\textup{marginal law of $\sigma_{va}$
		``in absence of $a$'',}\\
	\hq_e\equiv \hq_{av}
	\equiv \hq_{a\to v}
	&\equiv \textup{clause-to-variable message
		from $a$ to $v$}\\
	&=\textup{marginal law of $\sigma_{va}$
		``in absence of $v$''}.
	\end{align*}}%
(The messages $\dq,\hq$ represent distributions over warnings $\dmp,\hmp$; see \cite[Ch.~19]{MR2518205}.) It is well understood how to relate these messages by \textsc{bp} equations, which express the message outgoing from a variable (clause) across an edge as a function of the messages incoming to the variable (clause) across the other incident edges. These mappings are parametrized by the relevant weights:
	\begin{align}\nonumber
	\dq_{va}(\tau)
	&=
	\BP_{va}[ \hq;\Lm_v ]
		(\tau)
	\equiv \f1{\dot{z}_{va}}
		\sum_{\usi_{\delta v} : \sigma_{av}=\tau}
		\varphi_v(\usi_{\delta v};\Lm_v)
		\prod_{b\in\pd v\setminus a}
			\hq_{bv}(\sigma_{bv})\\
	\hq_{av}(\tau)
	&=
	\BP_{av}[ \dq;\Gm_a]
		(\tau)
	\equiv \f1{\hat{z}_{av}}
		\sum_{\usi_{\delta a} : \sigma_{av}=\tau}
		\hat{\varphi}_a(\usi_{\delta a};\Gm_a)
		\prod_{u\in\pd a\setminus v}
			\dq_{ua}(\sigma_{ua})
	\label{e:bp}
	\end{align}
where $\dot{z}_{va},\hat{z}_{av}$
are the normalizing constants
making the output of $\BP$ a probability measure.
We drop $\Lm,\Gm$ from the notation to indicate the unweighted recursions.
Although we generally do not write it explicitly,
the normalizing constants 
$\dot{z}_{va},\hat{z}_{av}$
also depend on the choice of weights.

For a leaf variable $v$ incident to a single edge $e=(av)$, recall from above that $\Lm_v$ simply denotes the edge weight on $e$. In this case, the variable-to-clause \textsc{bp} equation 
\eqref{e:bp}
simplifies to 
	\beq\label{e:leaf.bp}
	\dq_{va}(\tau)
	=\BP_{va}[\hq;\Lm_v](\tau)
	=\f{\lm_e(\tau)}{\sum_{\tau'} \lm_e(\tau')}\,,
	\eeq 
i.e., the message from $v$ into the graph is simply the weight on $v$. It follows, by recursing inwards from the leaves, that for any Gibbs measure $\nu$ on a finite tree, there is a unique solution $(\dq,\hq)$ of the \textsc{bp} equations. Local marginals of $\nu$ can be expressed in terms of this solution, for example, the marginal on the edges incident to a single vertex can be expressed as
	\begin{align}
	\label{e:var.tuple.measure.weighted}
	\text{(variable $v$)}\quad
	&\nu_{\delta v}(\usi_{\delta v})
	= \nu_{\delta v}
		[\Lm_v;\hq_{\pd v\to v}](\usi_{\delta v})
	\equiv \f1{\dbz_v}
	\varphi_v(\usi_{\delta v};\Lm_v)
	\prod_{a\in\pd v} \hq_{av}(\sigma_{av})\,,\\
	\text{(clause $a$)}\quad
	&\nu_{\delta a}(\usi_{\delta a})
	=\nu_{\delta a}[\Gm_a;
		\dq_{\pd a\to a}](\usi_{\delta a})
	\equiv
	\f1{\bm{\hat{z}}_a}
	\hat{\varphi}_a(\usi_{\delta a};\Gm_a)
	\prod_{v\in\pd a} \dq_{va}(\sigma_{av})\,.
	\label{e:clause.tuple.measure.weighted}
	\end{align}
Taking the edge marginal from either of these gives
	\beq\label{e:edge.marginal.q.times.q}
	\nu_{av}(\sigma_{av})
	= \nu_{av}[\dq_{va},\hq_{av}](\sigma_{av})
	\equiv
	\f{\dq_{va}(\sigma_{av}) \hq_{av}(\sigma_{va})}
	{\bar{z}_{av}}\,,
	\eeq
where the normalizing constant $\bar{z}_{av}$ satisfies the relations
	\beq\label{e:z.identity}
	\bar{z}_{av}
	= \f{\dbz_v }{ \dot{z}_{va}}
	= \f{ \bm{\hat{z}}_a }{ \hat{z}_{av}}\,.
	\eeq
The expressions \eqref{e:var.tuple.measure.weighted}, \eqref{e:clause.tuple.measure.weighted}, and \eqref{e:edge.marginal.q.times.q} will be used many times throughout the paper. In particular, for an edge $e=(av)$ we will often write ``$\nu_e\cong\dq_e\hq_e$'' to remind the reader of \eqref{e:edge.marginal.q.times.q}.

To conclude the section, we now explain how the tree recursions
for the frozen model (\S\ref{ss:wp.recursions}) can be retrieved as a special case of the color model \textsc{bp} recursions. This is the only place where we will make use of the warning propagation model, as an intermediary between the frozen model and color model (cf.\ \eqref{e:aux.model.bijections}). We shall keep the discussion here brief, and refer the reader to our previous works \cite{dss-naesat,MR3689942} where we covered analogous correspondences in substantial detail.

In \S\ref{ss:wp.recursions} we considered a finite tree $T$ with variables at the leaves, and defined the ``frozen model on $T$ with rigid boundary'' (Definition~\ref{d:frozen.model.bdy.conditions}). On each edge $(av)$ we defined a probability measure
$\bmeta_{va}$ on $\set{\plus,\minus,\free}$,
as well as a probability measure $\bm{\hat{u}}_{av}$ on $\set{\plus,\free}$.
We now demonstrate that this corresponds to the color model on $T$ with all vertices unweighted, except for the leaf variables where we put weights $\lm_v(\free)=0$.
As before, we use $q=(\dq,\hq)$ to denote the messages in this weighted color model. We now write $h=(\dh,\hh)$ for the analogous \textsc{bp} messages in the warning propagation model --- thus both $\dh_{va}$ and $\hh_{av}$
are probability measures over message pairs
$\msg_{av}\equiv(\dmp_{va},\hmp_{av})$.
Given $\bmeta$, we can define $h$ by setting $\dh_{va}(\dmp_{va},\hmp_{av})\cong\bmeta_{va}(\dmp_{va})$
and $\hh_{av}(\dmp_{va},\hmp_{av})\cong\bm{\hat{u}}_{av}(\hmp_{av})$. Working out the proper normalization gives
	\begin{align}\nonumber
	\dh_{va}(\dmp_{va},\hmp_{av})
	&\cong \f{\bmeta_{va}(\dmp_{va})}
		{ 2-\bmeta_{va}(\minus) }\,,\\
	\hh_{av}(\dmp_{va},\hmp_{av})
	&= \f{ \bm{\hat{u}}_{av}(\hmp_{av}) }
		{ 3-\bm{\hat{u}}_{av}(\plus) }
	\label{e:bp.sym}
	\end{align}
Projecting down to the color model, we can define
	\begin{align}\nonumber
	\bigg(
	\dq_{va}(\red),
	\dq_{va}(\yel),
	\dq_{va}(\grn),
	\dq_{va}(\blu)\bigg)
	&=\Bigg(
	\f{\bmeta_{va}(\plus)+\bmeta_{va}(\free)}
		{2-\bmeta_{va}(\minus)},
	\f{\bmeta_{va}(\minus)}{2-\bmeta_{va}(\minus)},
	\f{\bmeta_{va}(\free)}{2-\bmeta_{va}(\minus)},
	\f{\bmeta_{va}(\plus)}{2-\bmeta_{va}(\minus)}
	\Bigg)\,,\\
	\bigg(
	\hq_{av}(\red),
	\hq_{av}(\yel),
	\hq_{av}(\grn),
	\hq_{av}(\blu)\bigg)
	&=\Bigg(
	\f{\bm{\hat{u}}_{av}(\plus)}
		{3-2\bm{\hat{u}}_{av}(\plus)},
	\f{\bm{\hat{u}}_{av}(\free)}
		{3-2\bm{\hat{u}}_{av}(\plus)},
	\f{\bm{\hat{u}}_{av}(\free)}
		{3-2\bm{\hat{u}}_{av}(\plus)},
	\f{\bm{\hat{u}}_{av}(\free)}
		{3-2\bm{\hat{u}}_{av}(\plus)}
	\Bigg)\,.\label{e:color.recursions.eta}
	\end{align}
In the above, $2-\bmeta_{va}(\minus)$ and $3-\bm{\hat{u}}_{av}(\plus)$ are the normalizing constants. It is straightforward to verify that if $\bmeta$ and $\bm{\hat{u}}$ are as defined in \S\ref{ss:wp.recursions}, then $\dq$ and $\hq$ as defined by \eqref{e:color.recursions.eta} solve the \textsc{bp} equations \eqref{e:bp} for the color model on $T$ with weights $\lm_v(\free)=0$ at leaf variables $v$.\footnote{For the purposes of this paper, it suffices merely to note that \eqref{e:bp.sym} defines a valid \textsc{bp} solution for the warning propagation model, and projects to a \textsc{bp} solution for the color model. We omit the derivation of \eqref{e:bp.sym} since it is not central here, and similar correspondences were already explained in detail in \cite{dss-naesat,MR3689942}.} In particular, it follows from \eqref{e:leaf.bp} that the message $\dq_{va}$ from a leaf variable $v$ to its neighboring clause $a$ will be uniform over $\set{\red,\yel,\blu}$. This corresponds via \eqref{e:color.recursions.eta} precisely to the boundary conditions $\bmeta_{ua(u)}(\plus)=1/2=\bmeta_{ua(u)}(\minus)$, as specified in \S\ref{ss:wp.recursions}.

\section{Variable types, preprocessing, and proof outline}
\label{s:preprocess.defns}

In this section we formally describe the preprocessing algorithm and give a more detailed outline of the proof of the main result Theorem~\ref{t:main}. This section is organized as follows: 
\begin{enumerate}[--]
\item In \S\ref{ss:canonical} we define the notion of \bemph{simple types}, based on $R$-neighborhoods of variables in the original $\ksat$ graph. We also define \bemph{``canonical''} edge marginals based on $r$-neighborhoods of variables (for $r=R/10^4$), and define a \bemph{coherence} condition for edge marginals around a clause to be mutually compatible.

\item In \S\ref{ss:classification.of.simple.types}
we define a series of properties of simple types.

\item In \S\ref{ss:introprep} we define a procedure which takes a $\ksat$ instance $\GG'$, and produces a processed graph $\GG=\proc\GG'$ in which all variables satisfy certain desirable properties. The most important property of the processed graph is that it can be covered by regions (``enclosures'') where a certain estimate \eqref{e:every.variable.is.enclosed} holds.

\item In \S\ref{ss:contraction.overview} we give the basic outline of the proof of the main result
Theorem~\ref{t:main}. In this subsection we will state the first and second moment results that are needed for the proof.

\item In \S\ref{ss:first.moment.preliminaries} we present preliminary results towards the first moment calculation.

\item In \S\ref{ss:second.moment.overview} 
we present preliminary results towards the second moment calculation. This subsection also presents some of the conceptual ideas behind the second moment calculation.

\item In \S\ref{ss:coherence.weights} we give the technical details of the coherence condition from \S\ref{ss:canonical}.

\end{enumerate}

\subsection{Simple types and coherence}
\label{ss:canonical}

As mentioned above, a key step in~\cite{MR3436404} is to condition on the degree profile of~$\GG$. In this work we condition on the empirical distribution of depth-$R$ neighborhood types, which we regard as a ``generalized degree profile'' in the manner of~\cite{MR3405616}. We establish a satisfiability lower bound $\albd(R)$ for each fixed $R$, and show that $\albd(R)\to\arsb$ in the limit $R\to\infty$. To be precise, the order of limits taken throughout this paper is the following:
for each $k\ge k_0$ (where $k_0$ is a large absolute constant), take $n\to\infty$ with $R$ fixed, 
then take $R\to\infty$. To this end, we fix $r$ a (large) positive integer, and define $\rprime,R$ so that
	\beq\label{e:radii}
	\f{R}{10^4} = \f{\rprime}{10^2} = r\,.
	\eeq
Recall from Definition~\ref{d:formal.random.ksat} that we work with the measure $\poisP\equiv\poisP^{n,\alpha}$ on bipartite factor graphs $\GG=(V,F,E)$ (with edge signs) where $V = [n]\equiv\set{1,\ldots,n}$. We now further assign a \textbf{marking} of the variables, which is simply a uniformly random mapping
	\beq\label{e:random.R.marking}
	\LABEL \equiv \LABEL_R : V \to \bigg\{1,2,\ldots,
		\Big\lceil \exp(4^k R)\Big\rceil \bigg\}\,.\eeq
From now on, the graph $\GG$ is always understood to come equipped with the random marking $\LABEL_R$. (In the random graph $\GG\sim\P$, a typical variable has an $R$-neighborhood of volume at most $(k^{O(1)}2^k)^R$. The choice of $\lceil \exp(4^k R)\rceil$ for the range of the random markings ensures that in any given $R$-neighborhood, all variables receive distinct markings with chance $1-o_R(1)$.)

In the limit $n\to\infty$, the random graph $\GG\sim\poisP^{n,\alpha}$ converges locally in distribution (in the sense of \cite{MR1873300,MR2354165}, and reviewed in Remark~\ref{r:unimodular} below) to a Poisson Galton--Watson tree $\tree=(V_{\tree},F_{\tree},E_{\tree})$ which can be generated as follows: start with a single root variable $\vrt$, then generate offspring according to the rule that each variable independently generates $\Pois(\alpha k)$ child clauses, and each clause generates $k-1$ child variables. Each edge is labelled with a literal $\lit$ which takes values $\plus$ or $\minus$ with equal probability, independently over all the edges. Finally, for each variable $v\in V_{\tree}$ assign an independent uniformly random mark
	\[
	\LABEL(v) \equiv \LABEL_R(v) 
	\in
	\bigg\{1,2,\ldots,
		\Big\lceil \exp(4^k R)\Big\rceil \bigg\}\,.
	\]
We write $\uPGW^\alpha$ for the resulting probability measure on rooted trees.\footnote{The measure $\uPGW^\alpha$ differs from the most standard definition of the Poisson Galton--Watson law only in minor details: the alternation between variables and clauses, and the presence of random edge signs and variable marks.}

To generalize the notion of degree distribution, we begin with a preliminary definition of neighborhood type, as follows. This definition will be augmented over the course of this section.

\begin{dfn}[simple type]\label{d:simple.types} In a graph $\GG$, the \bemph{simple type} $t_e$ of a clause-variable edge $e\equiv(av)\equiv(va)\in E$ is the isomorphism class of $(B_R(v),e)$, the $R$-neighborhood around $v$ rooted at edge $e$.\footnote{The edge-rooted graphs $(T_i,e_i)$, $i=1,2$, are isomorphic if there is a bijective graph homomorphism $\iota : T_1 \to T_2$ which maps $e_1$ to $e_2$, preserves all edge labels $\lit_{av}\in\set{\plus,\minus}$ and indices $j(v;a)\in[k]$, and preserves all variable marks $\LABEL(v)
\in\set{1,\ldots, \exp(\lceil 4^k R\rceil) }$.} We write $j(t_e)\equiv j(v;a)$ to indicate the position of the edge in the clause. The \bemph{simple type} of a vertex $x \in V \cup F$ is the multi-set of simple types of all incident edges, $\set{t_e : e\in\delta x}$. Note that this has a different meaning depending on whether $x$ is a clause or a variable:
\begin{enumerate}[(i)]
\item If $x\in F$ is a clause, then its simple type --- which we hereafter denote $L_x$ --- is a multi-set with no repeated elements, since each edge $e\in\delta x$ has a distinct index $j(t_e)\in[k]$. Thus $L_x$ is equivalently represented as the ordered $k$-tuple $(L_x(1),\ldots,L_x(k))$ where $L_x(j)$ is the type of the $j$-th edge in $\delta x$.

\item If $x\in V$ is a variable, then
its simple type --- which we hereafter denote $T_x$ --- may have repeated elements. It is equivalently represented as the isomorphism class of $B_R(v)$ regarded as a graph rooted at $v$.
\end{enumerate}
We say that an edge $e$ is \bemph{acyclic} if its simple type $t_e$ is acyclic. We say that $e$ is \bemph{proper} if it is acyclic, and moreover no two variables $u\ne w$ in $t_e$ receive the same mark $\LABEL_R(u)=\LABEL_R(w)$. A vertex will be termed \bemph{acyclic} (resp.\ \bemph{proper}) if all its incident edges are acyclic (resp.\ proper). If $v\in V$ is a proper variable, then its simple type $T_v$ has no repeated elements.
\end{dfn}

\begin{rmk}\label{r:acyclic.proper.multiset}
In the graph $\GG$ sampled according to $\poisP\equiv\poisP^{n,\alpha}$, the fraction of cyclic variables will typically be around
	\[
	\f{2^{O(kR)}}{n}
	= o_n(1)\,
	\]
while the fraction of improper variables will typically be around
	\[
	\f{2^{O(kR)}}{\exp\{ 4^k R\}}
	= o_R(1)\,.\]
During processing we will remove all improper variables from the graph, ensuring that the final graph will have girth greater than $2R$, since all variables remaining will be proper (hence acyclic). This further ensures that for any surviving variable $v$, its simple type in the initial graph is a multi-set with no repeated elements.
\end{rmk}

\begin{dfn}[directed trees] \label{d:directed.trees}
If $T$ is any bipartite factor tree and $e=(av)$ is any edge in $T$:
\begin{enumerate}[--]
\item We let $T_{va}$ be the connected component of $T\setminus a$ that contains $v$. We regard $T_{va}$ as being rooted at $v$, where $v$ has parent edge $e$
that points to the deleted clause $a$. We call $T_{va}$ a \bemph{variable-to-clause tree}.
\item Similarly we let $T_{av}$ be the connected component of $T\setminus v$ that contains $a$. 
We regard $T_{av}$ as being rooted at $a$, where $a$ has parent edge $e$
that points to the deleted clause $v$. We call $T_{av}$ a \bemph{clause-to-variable tree}.
\end{enumerate}
We will refer to both $T_{va}$ and $T_{av}$ as \bemph{directed trees}. Note that trees of this kind have already appeared previously, in the discussions of \S\ref{ss:wp.recursions}.
\end{dfn}

\begin{dfn}[canonical messages and marginals]
\label{d:canonical} Recall the definition of $\FF$ from the discussion following \eqref{e:def.FF.of.tree}. For an acyclic edge $(av)\in E$, let $T=B_r(v)$, and let 
$\SQT=B_{r-1/2}(a)$ (equivalently, $\SQT$ is the union of $B_{r-1}(u)$ over $u\in\pd a$), so that $\SQT\subseteq T$ (and both are trees). Let 
	{\setlength{\jot}{0pt}
	\begin{align}\nonumber
	\etastar_{va}
	\equiv \FF_{T,va}
	= \FF_r(T_{va})
	&\longleftrightarrow \dqstar_{va}\,,\\
	\nonumber
	\bhustar_{av}
	\equiv\FF_{T,av}
	= \FF_{r-1/2}(T_{av})
	&\longleftrightarrow \hqstar_{av}\,,\\
	\nonumber
	\SQeta_{va}
	\equiv\FF_{\SQT,va}
	= \FF_{r-1}(\SQT_{va})
	= \FF_{r-1}(T_{va})
	&\longleftrightarrow\SQdq_{va}\,,\\
	\SQbhu_{av}
	\equiv\FF_{\SQT,va}
	= \FF_{r-1/2}(\SQT_{av})
	= \FF_{r-1/2}(T_{av})
	&\longleftrightarrow
	\SQhq_{av}\,,
	\label{e:defn.canonical.eta}
	\end{align}}
 where $\longleftrightarrow$ indicates the correspondence
\eqref{e:color.recursions.eta}. Thus
$\etastar_{va}$ and $\SQeta_{va}$ are probability measures over $\set{\plus,\minus,\free}$; and $\bhustar_{av}$ and $\SQbhu_{av}$ are probability measures over $\set{\plus,\free}$. Meanwhile $\dqstar_{va}$, $\hqstar_{av}$, $\SQdq_{va}$, and $\SQhq_{av}$ are all probability measures over $\set{\RYGB}$. Note that $\bhustar_{av}=\SQbhu_{av}$, so $\hqstar_{av}=\SQhq_{av}$. Recalling \eqref{e:edge.marginal.q.times.q}, we define
	\beq\label{e:defn.canonical.edge.marginal}
	\starpi_{av}(\sigma)
	\equiv\f{ \dqstar_{va}(\sigma)\,
		\hqstar_{av}(\sigma)}
	{ \sum_{\sigma'}
		 \dqstar_{va}(\sigma')\,
		\hqstar_{av}(\sigma')}\,,\quad
	\SQpi_{av}(\sigma)
	\equiv
	\f{ \SQdq_{va}(\sigma)\,
		\SQhq_{av}(\sigma)}
	{ \sum_{\sigma'}
		 \SQdq_{va}(\sigma')\,
		\SQhq_{av}(\sigma')}\,,
	\eeq
where in each case the sum in the denominator goes over $\sigma'\in\set{\RYGB}$. For an acyclic edge $e=(av)$, we call $\dqstar_{va}$ and $\hqstar_{av}$ the \bemph{canonical messages}, and $\starpi_{av}$ the \bemph{canonical marginal}, all based on $T=B_r(v)$. We call 
$\SQdq_{va}$ and $\SQhq_{av}$ the \bemph{clause-based messages}, and $\SQpi_{av}$ the \bemph{clause-based marginal}, all based on $\SQT=B_{r-1/2}(a)\subseteq T$.
\end{dfn}

\begin{rmk}\label{r:first.rmk.coherence}
Let $e=(av)$ be an acyclic edge, so $T=B_r(v)$
and $\SQT=B_{r-1/2}(a)\subseteq T$ are both trees. Note that the correspondence \eqref{e:color.recursions.eta}
implies the relations
	\begin{align*}
	\hqstar_{av}(\blu)
		=\hqstar_{av}(\grn)
		=\hqstar_{av}(\yel)
		&= \f{\bm{\hat{u}}_{av}(\free)}
		{3-2\bm{\hat{u}}_{av}(\plus)}
		\,,\\
	\dqstar_{va}(\red)
		=\dqstar_{va}(\blu)
		+\dqstar_{va}(\grn)
		&=
		\f{
		\bmeta_{va}(\plus)+\bmeta_{va}(\free)
		}{2-	\bmeta_{va}(\minus)}
		\,,
	\end{align*}
and similarly for $\SQhq$ and $\SQdq$. The canonical messages $\dqstar_{va}$ and $\hqstar_{av}$ (for $a\in\pd v$) are all based on $T$, which means they satisfy the \bemph{variable} \textsc{bp} relation $\dqstar_{va} = \BP_{va}[\hqstar]$. By contrast, they need \bemph{not} satisfy clause \textsc{bp} relations: even if all edges incident to clause $a$ are acyclic, it is \bemph{not} necessarily the case that $\hqstar_{av} = \BP_{av}[\dqstar]$, because for each $u\in\pd a\setminus v$ the message $\dqstar_{ua}$ is based on a different neighborhood $B_r(u)$. On the other hand, the clause-based messages $\SQdq_{ua}$ and $\SQhq_{au}$ (for $u\in\pd a$) are all based on $\SQT$, which means they do satisfy the \bemph{clause} \textsc{bp} relation $\SQhq_{av} = \BP_{av}[\SQdq]$. A related observation is that if we define (cf.\ \eqref{e:var.tuple.measure.weighted})
	\[
	\nu_{\delta v}(\usi_{\delta v})
	=\f1{\dbz_v}
	\varphi_v(\usi_{\delta v})
	\prod_{a\in\pd v}
	\hqstar_{av}(\sigma_{av})
	\]
where $\dbz_v$ denotes the normalization that makes 
$\nu_{\delta v}$ a probability measure, then $\nu_{\delta v}$ has edge marginals $\starpi_e$ for all $e\in\delta v$.
On the other hand, if we define
	\[\nu_{\delta a}(\usi_{\delta a})
	=\f1{\bm{\hat{z}}_a}
	\hat{\varphi}_a(\usi_{\delta a})
	\prod_{u\in\pd a}
	\SQhq_{ua}(\sigma_{au})\,,
	\]
then $\nu_{\delta a}$ has has edge marginals $\SQpi_e$ for all $e\in\delta a$. The measure $\nu_{\delta v}$ is consistent with the frozen model with rigid boundary conditions (Definition~\ref{d:frozen.model.bdy.conditions}) on $T$, while the measure $\nu_{\delta a}$ is consistent with the frozen model with rigid boundary conditions on $\SQT$.\end{rmk}

In the limit $r\to\infty$ we expect the difference between quantities based on $T=B_r(v)$ versus $\SQT=B_{r-1/2}(a)$ to go away. However, when working at a fixed finite radius, a major technical difficulty is to handle inconsistencies between the two. In this paper we work primarily with the quantities based on $T=B_r(v)$ (hence our choice of the term ``canonical'' to describe those quantities). As discussed above, the canonical messages do not, in general, satisfy clause \textsc{bp} relations. A closely related issue is that for $u\in\pd a$, each $\starpi_{au}$ is based on a different neighborhood $B_r(u)$. An important technical result of this section (stated and proved in \S\ref{ss:coherence.weights}) shows that we can \bemph{reweight} clauses such that the \textsc{bp} equations do hold exactly --- provided the clauses are ``not excessively inconsistent,'' as we now formalize. Following \cite{MR3436404}, it is useful to 
define \bemph{composite colors}
	{\setlength{\jot}{0pt}\begin{align}\nonumber
	\cya &\equiv\SPIN{cyan}\equiv\set{\grn,\blu}\,,\\
	\pur &\equiv\SPIN{purple}
		\equiv\set{\red,\blu}\,.
	\label{e:cyan.purple}
	\end{align}}%
The clause factor \eqref{e:indicator.of.valid.clause.coloring} does not distinguish between $\grn$ and $\blu$, while the variable factor \eqref{e:EVAL.of.color} does not distinguish between $\red$ and $\blu$; introducing the composite colors helps to simplify some parts of the analysis. The following gives our formal criterion for ``consistency'' within clauses:

\begin{dfn}[coherence]\label{d:coherence}
For a clause $a\in F$, suppose $\pi=(\pi_e)_{e\in\delta a}$ where each $\pi_e$ is any probability measure over $\set{\RYGB}$. We say that $\pi$ is \bemph{weakly coherent} if it satisfies the following:
	\beq\label{e:cohere.y.for.r}
	\COHER_e(\pi) \equiv
	\pi_e(\yel)
	- \sum_{e' \in\delta a \setminus e}
	\pi_{e'}(\red) \ge 0\,,
	\eeq
for each $e\in\delta a$; and with $\pi_e(\cya) \equiv \pi_e(\blu)+\pi_e(\grn)$ we have
	\beq\label{e:cohere.enough.cyan}
	\COHER_a(\pi)
	\equiv \sum_{e\in\delta a}
	\pi_e(\cya) - 2\bigg\{
	1 - \sum_{e\in\delta a}\pi_e(\red)
	\bigg\} \ge 0\,.\eeq
We say that $\pi$ is \bemph{strictly coherent} if condition~\eqref{e:cohere.enough.cyan} holds with strict inequality, and condition~\eqref{e:cohere.y.for.r} holds with strict inequality for each $e\in\delta a$ where $\pi_e(\yel)>0$. We then say that the clause $a$ is \bemph{weakly (strictly) coherent} if it is acyclic, and its canonical edge marginals $\starpi=(\starpi_e)_{e\in\delta a}$ are weakly (strictly) coherent. 
\end{dfn}

 A full analysis of the (weak and strict) coherence conditions is deferred to \S\ref{ss:coherence.weights}. For the purposes of the upcoming discussion, the most immediately relevant result from \S\ref{ss:coherence.weights} is Lemma~\ref{l:clause.based.marginals.cohere} which says that $\SQpi$ is strictly coherent for all $r\ge2$, and so $\starpi$ will also be strictly coherent if it is ``close enough'' to $\SQpi$. The remainder of \S\ref{ss:coherence.weights} is occupied with showing that strictly coherent measures can be realized by appropriate reweighting systems (in particular, see Corollaries~\ref{c:cohere.weights}~and~\ref{c:clause.bp.weights}).

\subsection{Classification of simple types}
\label{ss:classification.of.simple.types}

In this subsection we make several classifications of (simple, acyclic) types in order to identify the vertices that must be removed during preprocessing. Recall \eqref{e:radii}. One aim will be to ensure that all clauses in the final processed graph
are strictly coherent (Definition~\ref{d:coherence}). For an acyclic edge $e=(av)$, recall 
the canonical messages and marginals
of Definition~\ref{d:canonical}
(\eqref{e:defn.canonical.eta}~and~\eqref{e:defn.canonical.edge.marginal}), all based on the neighborhood $B_r(v)$. Since all these quantities can be determined from the edge type $t=t_e$ (which encodes the structure of the larger neighborhood $B_R(v)$), we will freely interchange $e$ and $t$ in the subscripts, so for instance $\starpi_e \equiv \starpi_t$.

\begin{dfn}[stable]\label{d:stable}
Suppose the edge $e=(av)$ is acyclic (Definition~\ref{d:simple.types}), so that all the messages and marginals of Definition~\ref{d:canonical} are well-defined. As we noted above --- and will prove in Lemma~\ref{l:clause.based.marginals.cohere} below --- the measure $\SQpi$ is strictly coherent, i.e., it satisfies all the conditions of Definition~\ref{d:coherence}. We say that the edge $e=(av)$ is \bemph{marginal-stable} if all the following bounds hold:
	\begin{align}
	\label{e:stable.yellow}
	\Big|\starpi_e(\yel)-\SQpi_e(\yel)\Big| 
	&\le \f{\COHER_e(\SQpi)}{5k}\,,\\
	\label{e:stable.cyan}
	\Big|\starpi_e(\cya)-\SQpi_e(\cya)\Big|
	&\le \f{\COHER_a(\SQpi)}{5k}\,,\\
	\label{e:stable.red}
	\Big|\starpi_e(\red)-\SQpi_e(\red)\Big|
	&\le \f1{5k}
	\min\bigg\{
	\COHER_a(\SQpi),
	\min_{e'\in\delta a \setminus e} \COHER_{e'}(\SQpi)
	\bigg\}\,,\\
	\starpi_e(\sigma) 
	&\ge \f12 \SQpi_e(\sigma)
	\quad\textup{for all }
	\sigma\in\set{\red,\yel,\blu,\cya}\,.
	\label{e:stable.positivity}
	\end{align}
(We impose the last condition \eqref{e:stable.positivity} because it forces $\supp\starpi_e=\supp\starpi_e$, which is convenient for the analysis.) We say that $e$ is \bemph{message-stable} if
	\beq\label{e:message.stability}
	1-\etastar_{va}(\minus)\ge \f1{2^r}\,,\quad
	\max_{\sigma\in\set{\red,\blu,\yel\grn}} \bigg|
	\f{\hqstar_{av}(\sigma)}{\BP_{av}[\dqstar](\sigma)}
	-1\bigg| \le \f1{k^r}\,.
	\eeq
(We take the convention $0/0=1$, so in the above it is permitted to have
 both $\hqstar_{av}(\sigma)$ and $\BP_{av}[\dqstar](\sigma)$ equal to zero.) We say that the edge is \bemph{stable} if it is both marginal- and message-stable. Finally, we say that an acyclic variable $v$ is \bemph{stable} if all its incident edges $e\in\delta v$ are stable. 
\end{dfn}

\begin{rmk*} Note that if all the variables in a graph are marginal-stable, then $\starpi$ will be strictly coherent on all the clauses in the graph. Indeed, substituting \eqref{e:stable.yellow} 
and \eqref{e:stable.red} into
\eqref{e:cohere.y.for.r} gives
	\[
	\COHER_e(\starpi)
	\ge \bigg(1 - \f15 \bigg) \COHER_e(\SQpi)\,,
	\]
for all edges $e$ in the graph. Similarly, substituting \eqref{e:stable.cyan} and \eqref{e:stable.red} into \eqref{e:cohere.enough.cyan} gives
	\[\COHER_a(\starpi)
	\ge\bigg(1 - \f35 \bigg) \COHER_a(\SQpi)\]
for all clauses $a$ in the graph. Since $\SQpi$ is strictly coherent (Lemma~\ref{l:clause.based.marginals.cohere}), it follows that $\starpi$ is also.\end{rmk*}

\begin{dfn}[nice] \label{d:nice}
An acyclic variable $v$ is \bemph{nice} if it has degrees 
	\beq\label{e:nice.degree.condition}
	\bigg||\pd v(\PM)| - \f{2^k k\log 2}{2}\bigg|
	\le 2^{2k/3}\,,\eeq
and its incoming and outgoing canonical messages satisfy the bounds
	\begin{align}\label{e:nice.var.message.cond}
	\max\Bigg\{
	\bigg|\dqstar_{va}(\yel)-\f13\bigg|,
	\bigg|2^k \dqstar_{va}(\grn)-\f13\bigg|
	\Bigg\}
	&\le \f1{2^{k/10}}\,,\\
	\label{e:nice.clause.message.cond}
	\max\Bigg\{
	\bigg|\hqstar_{av}(\yel)-\f13\bigg|,
	\bigg|2^{|\partial a|-1}
	\hqstar_{av}(\red)-\f13\bigg|
	\Bigg\}
	&\le \f1{2^{k/10}}\,,
	\end{align}
for all $a\in\pd v$.\footnote{Recall that we automatically have the identities $\hqstar_{av}(\yel)=\hqstar_{av}(\grn)=\hqstar_{av}(\blu)$ and $\dqstar_{va}(\red)=\dqstar_{va}(\blu)+\dqstar_{va}(\grn)$. Thus the conditions \eqref{e:nice.var.message.cond}
and \eqref{e:nice.clause.message.cond}
constrain $\dqstar_{va}(\sigma)$
and $\hqstar_{av}(\sigma)$ for all $\sigma\in\set{\RYGB}$.} (In condition \eqref{e:nice.clause.message.cond} we write $|\pd a|$, rather than simply $k$, because we will also use this definition in the processed graph where some clauses will have degree $k-1$.) Since the canonical messages are functions of $B_r(v)$, niceness is also a property of $B_r(v)$.
\end{dfn}

\begin{dfn}[$1$-stable and $1$-nice]
\label{d:j.stable}
We say an acyclic variable $v$ is \bemph{$1$-stable} if is stable, and remains stable after the removal of any one subtree descended from a variable $u \in B_r(v)\setminus v$. The \bemph{$1$-nice} property is analogously defined. We use \bemph{$0$-nice} to mean simply \bemph{nice}. For $\II\in\set{0,1}$ we let
	\beq\label{e:not.j.nice}
	D^{*,\II}
	\equiv 
	\bigg\{v\in V:
	v \text{ is acyclic but not $\II$-nice}\bigg\}\,.
	\eeq
Note that $D^{*,0}\subseteq D^{*,1}$.
\end{dfn}

We will now identify defective regions via 
the following bootstrap percolation process. In a general bipartite factor graph $\GG=(V,F,E)$, given some subset of variables $D_0\subseteq V$, for $t\ge1$ set $D_t\supseteq D_{t-1}$ to be the union of $D_{t-1}$ together with all variables having at least two neighboring variables in $D_{t-1} \cap V$. The set
	\beq\label{eq-bootstrap-percolation}
	\bsp(D_0;\GG)
	\equiv \bigcup_{t\ge0} D_t
	\eeq
will be termed
the \bemph{bootstrap percolation} of $D_0$ in $\GG$.

\begin{dfn}[defective]\label{d:j.defective}
Let $\KAPPA$ be a large absolute constant, to be determined later.\footnote{Ultimately we will require $\KAPPA\ge(240/\zeta)^4$
(see Propositions~\ref{p:contraction.COMPOUND} and \ref{p:nondefect.join}) 
where $\zeta$ is the absolute constant from Definition~\ref{d:error.notation.edge}. Finally we will take $\zeta=\ZETA/4$ where $\ZETA$ is the absolute constant from an \textit{a~priori} estimate, Proposition~\ref{p:apriori}.} Recall \eqref{e:radii} that $\rprime= 10^2 r$. 
Recall \eqref{e:not.j.nice} that 
$D^{*,\II}$ denotes the set of variables that are acyclic but not $\II$-nice, for $\II\in\set{0,1}$.
Let $D^{\KAPPA,\II} \equiv B_{\KAPPA}(D^{*,\II})$ be the union of $B_{\KAPPA}(v)$ over all $v\in D^{*,\II}$. The set of \bemph{$\II$-defective} variables is defined as
	\[\DEFECTIVE^\II
	\equiv\DEFECTIVE^\II(\GG)
	\equiv
	\bigg\{ v :
	v\in \bsp\Big( D^{\KAPPA,\II}
		\cap B_{\rprime/2}(v); 
		B_{\rprime/2}(v)\Big)\bigg\}\,.
	\]
We use \bemph{0-defective} and \bemph{defective} interchangeably; note $\DEFECTIVE^0(\GG)\subseteq\DEFECTIVE^1(\GG)$. Whether a variable is $\II$-defective can be determined from its $\rprime$-neighborhood. We say that a clause $a$ is \bemph{$\II$-defective} if all its incident variables are $\II$-defective. A $\II$-\bemph{defect} of $\GG$ is a (maximal) connected component of $\II$-defective variables and clauses. Note that each $\II$-defect has at its boundary a buffer of nice variables of depth at least $\KAPPA$.
\end{dfn}

For the remainder of the paper, we say \bemph{path} (in the bipartite factor graph $G$) to mean a finite sequence of vertices
	\[P\equiv
	\bigg(v_1,a_2,v_2,\ldots,a_\ell,v_\ell\bigg)\]
with no repeated elements, such that each entry of the sequence is a neighbor (in $G$) of the previous entry.
The path $P$ contains $\ell$ variables, and has \bemph{length} $\ell-1$. For acylic variables $u,v\in V$ at distance $d(u,v)\le R$, and $\II\in\set{0,1}$,
let $\mathfrak{B}^\II(u,v)$ count the variables on the (unique) shortest path from $u$ to $v$ (inclusive) that are $\II$-defective. The following property is essential to our contraction argument.

\begin{dfn}[contained]\label{d:contained}
Let $\DELTACONST$ be a small absolute constant, to be determined later.\footnote{Ultimately we will require $\DELTACONST \le \min\set{\zeta/30, (\log 2)/(2(\KAPPA)^{1/2})}$ (see Proposition~\ref{p:contraction.COMPOUND})
where $\zeta$ is the absolute constant from Definition~\ref{d:error.notation.edge}. Finally we will take $\zeta=\ZETA/4$ where $\ZETA$ is the absolute constant from Proposition~\ref{p:apriori}.} For an acyclic variable $v$, for $\II\in\set{0,1}$, and for any integer $1\le t\le 2R'$, we define
	\beq\label{e:def.R.for.containment}
	\mathfrak{R}^\II(v,t)
	\equiv
	\sum_{u : t \le d(u,v) < 2\rprime}
	\f{
	\exp\{ k(\DELTACONST)^{-1} 
	\mathfrak{B}^\II(u,v) \} }
	{ \exp\{
		(k\log2)(1+\DELTACONST)
		d(u,v) 
		\}}\,.\eeq
Note that $\mathfrak{R}^\II(v,2R')$ is zero, since it is an empty sum. We then define the \bemph{$\II$-containment radius} of $v$ to be
	\beq\label{eq-def-rad}
	\rad^\II(v)
	\equiv
	\min \bigg\{ t\ge 1:
	\mathfrak{R}^\II(v,t)
	\le \f14 \bigg\}
	\le 2R' \,,
	\eeq
so that $1 \le \rad^\II(v) \le 2\rprime$. For $\II\in\set{0,1}$, we say $v$ is \bemph{$\II$-self-contained} if we have $\rad^\II(u)\le d(u,v)$ for all $u$ with $1\le d(u,v)\le\rprime$. Note that the containment radius of any variable can be determined from its $3\rprime$-neighborhood. Whether a variable is $\II$-self-contained can be determined from its $4\rprime$-neighborhood.
\end{dfn}

The central aim of preprocessing is to ensure that it is possible to carve up the graph into ``enclosures'': the formal definition is given below, but roughly speaking these will be regions of diameter at most $\rprime$ such that every variable in a given enclosure has containment radius less than or equal to its minimal distance from the enclosure boundary (in particular, all the boundary variables must be self-contained). At the same time we will require another desirable property, which will be applied in the proof of Proposition~\ref{p:sep} below: 

\begin{dfn}[orderly]\label{d:orderly}
For $\II\in\set{0,1}$, we say that an acyclic variable $v$ is $\II$-\bemph{orderly} if along any path emanating from $v$ of length at most $\rprime$, at most $(\DELTACONST)^3$ fraction of variables along the path are $\II$-defective. In particular, a path leaving $v$ of length less than $(\DELTACONST)^{-3}$ cannot contain any $\II$-defective variable: that is to say, an $\II$-\bemph{orderly} variable cannot lie within distance $(\DELTACONST)^{-3}$ of any defect. Since being $\II$-defective is a property of a variable's $\rprime$-neighborhood, being $\II$-orderly is a property of a variable's $2\rprime$-neighborhood.
\end{dfn}

The following definitions will be of use in carving up the graph:

\begin{dfn}[perfect; fair]\label{d:perfect.fair}
For $\II\in\set{0,1}$,
we say that an acyclic variable $v$ is \bemph{$\II$-perfect} if it is both $\II$-orderly and $\II$-self-contained; this is a property of the variable's $4\rprime$-neighborhood. We say $v$ is \bemph{$\II$-fair} if 
\begin{enumerate}[(i)]
\item it is $\II$-stable;
\item\label{i:fair.volume.bound} its $5\rprime$-neighborhood contains no more than $\exp\{k^2 (5\rprime)\}$ variables; and
\item\label{i:fair.path.req} every length-$\rprime$ path emanating from $v$ contains at least one $\II$-perfect variable.
\end{enumerate}
Whether a variable is $\II$-fair is a property of its $5\rprime$-neighborhood.
\end{dfn}

Lastly we wish to ensure that every type appears a linear number of times in the preprocessed graph. The following definition is towards this purpose.

\begin{dfn}[good; excellent]\label{d:good.exc}
A rooted tree $T$ of depth $10\rprime$ is \bemph{$\II$-excellent} if, with
$\cong$ denoting isomorphism of rooted graphs,
and with $\uPGW\equiv\uPGW^\alpha$ as defined above (prior to Definition~\ref{d:simple.types}), we have 
	\beq\label{e:j.excellent}
	p_{\textup{ex},\II}(T)\equiv
	\uPGW\left( 
	\left.\begin{array}{c}
	B_{20\rprime}(u)
	\text{ contains 
	any}\\
	\text{variable which is not $\II$-fair}
	\end{array}
	\, \right| \, B_{10\rprime}(u) \cong T
	\right) \le \f1{\exp(k^3 R)}\,.
	\eeq
An acylic variable $v$ is termed \bemph{$\II$-excellent} if its $10\rprime$-neighborhood $B_{10\rprime}(v)$ satisfies \eqref{e:j.excellent}. Lastly, we say that $v$ is \bemph{$\II$-good} if (i) it is $\II$-fair, and (ii) every length-$40\rprime$ path emanating from $v$ contains at least one $\II$-excellent variable. Note that $\II$-excellent implies $\II$-good which in turn implies $\II$-fair. \end{dfn}

\subsection{Preprocessing algorithm}
\label{ss:introprep}

We now formally describe our preprocessing procedure, which depends on the parameter $R$. Recall from \eqref{e:radii} that $R = 10^2 \rprime = 10^4 r$. Recall that we write $B_\ell(v)$ to indicate the $\ell$-neighborhood of variable $v$. In what follows, we will often write
$B_\ell(v;\GG)$ to emphasize that the $\ell$-neighborhood is defined with respect to the graph $\GG$.


\begin{dfn}[removal process $\bsp'$] \label{d:proc}
In a graph $\GG=(V,F,E)$, define the \textbf{activated set}
	\beq\label{e:bspprime.activated.set}
	\ACT(\GG)
	\equiv \left\{
	\begin{array}{c}
	\text{variables $v\in V$ such that
	$B_{3R/10}(v;\GG)$ contains}\\
	\text{at least two clauses of degree $k-1$, or}\\
	\text{at least one clause of degree
	$\le k-2$.}
		\end{array}\right\}\,.\eeq
Given an initial subset of variables $A\subseteq V$, let $\GG_{-1}\equiv\GG$, $A_{-1}\equiv A$, and denote
	\[\GG_0 \equiv \proc_0 \GG \equiv\GG
	\,\bigg\backslash\,
	\Bigg\{ 	\bigcup_{v\in A} B_R(v;\GG)\Bigg\}
	= \GG
	\,\bigg\backslash\, B_R(A;\GG)
	\,.\]
Then, for $t\ge0$, we use \eqref{e:bspprime.activated.set} to define inductively $A_t \equiv \ACT(\GG_t)$ and
	\[\GG_{t+1} \equiv \proc_{t+1}\GG \equiv \GG_t
	\,\bigg\backslash\,
	\Bigg\{\bigcup_{v\in A_t}B_R(v; \GG_t)\Bigg\}
	= \GG_t \,\bigg\backslash\, B_R(A_t;\GG_t)
	\,.\]
The process terminates at the first time $t$ that $\ACT(\GG_t)=\emptyset$, and we denote the terminal graph $\GG_t \equiv \GG_\infty$. We let $\bsp'(A;\GG)$ denote the set of all variables removed by this procedure,
	\[\bsp'(A;\GG)
	\equiv
	\bigcup_{t = -1}^\infty
	\Bigg\{ \bigcup_{v\in A_t}B_R(v;\GG_t) \Bigg\}
	= \bigcup_{t = -1}^\infty B_R(A_t;\GG_t)\,.
	\]
With a minor abuse of notation we will also let $\bsp'(A;\GG)$ denote the subgraph of $\GG$ induced by these variables.
We hereafter refer to $\bsp'(A;\GG)$
as the \bemph{removal process in $\GG$ with initial set $A$}.
We then define:
	\[
	\begin{array}{l} 
	\hline \vspace{-10pt} \\
	\textbf{Preprocessing algorithm 
		on $\GG=(V,F,E)$:}\\
	\text{Let $A\subseteq V$ be the set of all 
	variables in $\GG$ that are}\\
	\text{improper (Definition~\ref{d:simple.types}) or
	 not $1$-good (Definition~\ref{d:good.exc}).} \\
	\text{Delete $\bsp'(A;\GG)$
	and output the processed graph
	$\proc\GG \equiv \GG_\infty$.}
	\vspace{4pt}
	\\ \hline
	\end{array}
	\]
Throughout this paper, we use the terms ``preprocessing procedure'' or ``preprocessing algorithm'' to refer to this mapping $\GG\mapsto\proc\GG$.
\end{dfn}

The processed graph $\proc\GG$ is guaranteed to have certain desirable properties, as we now describe. First note that if a variable $v$ survives in $\proc\GG$, then its processed neighborhood $B_{3R/10}(v;\proc\GG)$ must be obtainable by deleting at most one subtree from its original neighborhood $B_{3R/10}(v;\GG)$. (Otherwise, $v$ would belong to $\ACT(\proc\GG)$, a contradiction since $\ACT(\proc\GG)$ is by definition empty.) This means that if $v$ was $1$-nice in the original graph $\GG$ and survives in the processed graph $\proc\GG$, then it must be nice (i.e., $0$-nice) with respect to $\proc\GG$. Recalling \eqref{e:not.j.nice}, this directly implies $D^{*,0}(\proc\GG)\subseteq D^{*,1}(\GG)$. Recalling Definition~\ref{d:j.defective}, taking $\KAPPA$-neighborhoods gives 
	\[D^{\KAPPA,0}(\proc\GG))
	=B_{\KAPPA}\Big( D^{*,0}(\proc\GG);\proc\GG\Big)
	\subseteq B_{\KAPPA}\Big(D^{*,1}(\GG);GG\Big)
	= D^{\KAPPA,1}(\GG))\,.\]
If $D\subseteq D'$ and $\HH\subseteq\HH'$ then
$\bsp(D;\HH)\subseteq\bsp(D';\HH')$, so it follows that
	\beq\label{e:defective.pre.and.post}
	\DEFECTIVE^0(\proc\GG)
	=\bigg\{v\in \proc\GG:
		v\in 
		\bsp\bigg(
		D^{\KAPPA,0}(\proc\GG)
		\cap B_{\rprime/2}\Big(v;\proc\GG\Big); 
		B_{\rprime/2}\Big(v;\proc\GG\Big)
		\bigg)
	\bigg\}
	\subseteq
	\DEFECTIVE^1
	(\GG)\,.\eeq
This means that if a graph is not $1$-defective in $\GG$, then it is not defective in $\proc\GG$.

The inclusion \eqref{e:defective.pre.and.post} has the following implications. First, recalling Definition~\ref{d:orderly}, it follows that if a variable $v$ was $1$-orderly in the original graph $\GG$, and survives to the processed graph $\proc\GG$, then it is orderly with respect to $\proc\GG$. Further, recalling \eqref{e:def.R.for.containment}
from Definition~\ref{d:contained}, it holds for all $t$ that
	\[\mathfrak{R}^0\Big(v,t;\proc\GG\Big)
	\le\mathfrak{R}^1\Big(v,t;\GG\Big)\,.\]
Consequently, recalling \eqref{eq-def-rad}, 
any variable $v$ in $\proc\GG$ must satisfy
	\beq\label{e:post.radius.vs.pre.radius}
	\rad^0\Big(v;\proc\GG\Big)
	\le \rad^1\Big(v;\GG\Big)\,.\eeq
In particular, this implies that if a variable $v\in\proc\GG$ was $1$-self-contained in $\GG$, then it is self-contained in $\proc\GG$. It follows that if $v\in\proc\GG$ was $1$-perfect in $\GG$, then it is perfect in $\GG$.

Since $A$ includes all improper variables (Definition~\ref{d:simple.types}), hence all acyclic variables, the processed graph $\proc\GG$ will have girth greater than $2R$. By the same reasoning as for the $1$-nice property, 
it follows from Definition~\ref{d:j.stable} that if $v\in\proc\GG$ was $1$-stable in $\GG$, then it is stable in $\proc\GG$. It then follows from Definition~\ref{d:perfect.fair} that if $v\in\proc\GG$ was $1$-fair in $\GG$, then it is fair in $\proc\GG$. Since the processing removes all variables that are not $1$-good (Definition~\ref{d:good.exc}), any $v\in\proc\GG$ must have been $1$-good in $\GG$, hence also $1$-fair. It follows that \textbf{all variables in $\proc\GG$ are fair} (with respect to $\proc\GG$). It then follows by Definition~\ref{d:perfect.fair} part~\eqref{i:fair.path.req} that any connected component in $\proc\GG$ of non-perfect variables must have diameter at most $\rprime$, so in particular it must be a tree. We can therefore carve up the graph into small regions separated by perfect variables:

\begin{dfn}[compound enclosure]\label{d:enclosure} A \bemph{compound enclosure} is a subgraph $U\subseteq\proc\GG$ induced by a subset of variables $U^\circ \cup \pd U^\circ\subseteq\proc V$
where $U^\circ$ is a nonempty, maximal connected component of variables that are not perfect in $\proc\GG$, and
	\[\pd U^\circ \equiv \pd_\circ U \equiv
	\bigg\{u\in\proc V : d(u,U^\circ)=1\bigg\}
	\]
is the external boundary of $U^\circ$ (which, by definition, must consist entirely of variables that are perfect in $\proc\GG$). By the above observations, the compound enclosure must be a tree of diameter at most $\rprime$.
\end{dfn}

By construction, all compound enclosures are pairwise edge-disjoint. It is possible for two compound enclosures to share one boundary variable; they cannot have more than one variable in common because $\proc\GG$ has girth at least $2R$. In any compound enclosure $U$,
if $u\in \pd_\circ U$ and $v\in U^\circ$ then it follows from 
\eqref{e:post.radius.vs.pre.radius} together with the self-containment condition on $u$ (Definition~\ref{d:contained}) that
	\[
	\rad^0\Big(v;\proc\GG\Big)
	\le\rad^1\Big(v;\GG\Big)
	\le d\Big(v,u ; \GG\Big)
	\le d\Big(v,u;\proc\GG\Big)\,.
	\]
Minimizing the right-hand side over all $u\in\pd_\circ U$ gives
	\beq\label{e:every.variable.is.enclosed}
	\rad^0\Big(v;\proc\GG\Big)
	\le d\Big(v, \pd_\circ U; \proc\GG\Big)
	\eeq
for all $v\in U^\circ$. The property
\eqref{e:every.variable.is.enclosed} is essential to our analysis for compound enclosures.

\begin{dfn}[simple total type]
For each edge $e=(av)$ in the processed graph $\proc\GG$, the \bemph{simple total type} of the edge records its simple type (Definition~\ref{d:simple.types}) both before and after preprocessing: that is to say, it is the ordered pair of edge-rooted trees
	\[
	\bigg( \Big(B_R(v;\GG),e\Big),
		\Big(B_R(v;\proc\GG),e\Big) \bigg)\,,
	\]
modulo (edge-rooted) isomorphism. (As discussed in a footnote to Definition~\ref{d:simple.types}, the isomorphism must also preserve edge labels $\lit_{av}$, indices $j(v,a)$, and marks $\LABEL(v)$.)
\end{dfn}

\begin{dfn}[compound type]
\label{d:cpd.type}
For any edge $e$ appearing inside a compound enclosure $U$ of the processed graph $\proc\GG$, the \bemph{compound type} of the edge records the graph structure of $U$ with the position of $e$ marked, as well as the simple total type of every edge $e'$ in $U$. In particular, different edges appearing in the same compound enclosure $U$ must have different compound types (since they take different positions in $U$), even if their simple total types match. 
\end{dfn}

\begin{dfn}[total type]\label{d:total.type} The \bemph{total type} $\bm{t} \equiv \bm{t}_e$ of an edge $e$ in $\proc\GG$ is defined to be its compound type if it belongs to a compound enclosure, and its simple total type otherwise. The \bemph{total type} of a variable or clause in $\proc\GG$ is the multi-set of incident edge total types --- since improper variables were removed during processing, the multi-sets are now simply sets. We hereafter use $\bT$ to denote variable total types, and $\bL$ to denote clause total types. For an edge $e=(av)$ of type $\bm{t}$ we write $j(\bm{t})\equiv j(v;a)$ for the position of the variable within the clause. We write $\bt\in\bL$ to indicate compatibility of types in the sense that $\bL(j(\bm{t})) = \bm{t}$.
\end{dfn}

\begin{dfn}[processed neighborhood profile]\label{d:degseq}
Let $\GG'=(V',F',E')$ be a $\ksat$ problem instance, and denote its processed version by $\GG=(V,F,E)=\proc\GG'$. Write $V=\set{v_1,\ldots,v_n}$ and $F=\set{a_1,\ldots,a_m}$. We then define the \bemph{processed neighborhood sequence of $\GG'$} (equivalently, the \bemph{neighborhood sequence of $\GG$}) as
	\[
	\cD
	\equiv
	\cD_{\GG}
	\equiv
	\begin{pmatrix}
	(\bT_{v_1},\ldots,\bT_{v_n})\\
	(\bL_{a_1},\ldots,\bL_{a_m})
	\end{pmatrix}\,.
	\]
We then define $\DD$ to be the same as $\cD$ except that we forget the ordering of $V$ and of $F$. Thus $\DD$ contains only the information of the empirical counts
	{\setlength{\jot}{0pt}\begin{align}
	n_{\bT}
	&\equiv
	\textup{number of variables in $\proc V$ of total type $\bT$;} \nonumber \\
	m_{\bL}
	&\equiv
	\textup{number of clauses in $\proc F$ of total type $\bL$;} \nonumber \\
	n_{\bm{t}}
	&\equiv
	\textup{number of edges in $\proc E$ of total type $\bm{t}$.}
	\label{e:empirical.counts.of.total.types}
	\end{align}}%
The empirical count of edge types can be determined as a marginal of either the clause or variable counts:
	\beq\label{e:edge.marginal.of.types}
	\sum_{\bT}
	n_{\bT} \Ind{\bt \in \bT}
	= n_{\bt}
	= \sum_{\bL}
	m_{\bL}
	\sum_j
	\Ind{\bL(j)=\bt}\,.
	\eeq
(Recall from Remark~\ref{r:acyclic.proper.multiset} that the simple type of any variable which survives preprocessing is a multi-set with no repeated elements --- thus, in \eqref{e:edge.marginal.of.types}, $\Ind{\bm{t}\in\bT}$ is the same as the number of occurrences of $\bm{t}$ in $\bT$.) We will sometimes abuse notation and use $\DD$ to denote the normalized empirical measures
	\[
	\dot{\DD}(\bT)
	= \f{n_{\bT}}{|V|},\quad
	\hat{\DD}(\bL)
	= \f{m_{\bL}}{|F|},\quad
	\bar{\DD}(\bm{t})
	= \f{n_{\bm{t}}}{|E|}\,.
	\]
However, we emphasize that $\DD$ encodes $|V|=n$ and $|F|=m$ in addition to the normalized empirical measures. We call $\DD\equiv\DD_{\GG}$ the \bemph{processed neighborhood profile of $\GG'$}
(equivalently, the 
\bemph{neighborhood profile of $\GG$}).
\end{dfn}

The next remark is most relevant to the precise statement of Proposition~\ref{p:unif} below:

\begin{rmk}\label{r:processed.graph.types.CM} 
The processed graph $\GG=\proc\GG'$ is a subgraph of the original graph $\GG'$, and therefore carries less information in general. For the rest of this paper, whenever we refer to the processed graph $\GG$, we assume that each vertex and each edge of $\GG$ \textbf{carries the information of its total type}. In particular, each variable in $\GG$ does carry the information concerning its simple type (Definition~\ref{d:simple.types}), even if part of its original $R$-neighborhood is deleted during processing. However, $\GG$ does \textbf{not} carry the information of the entire original graph $\GG'$, so we still do allow for the possibility that $\proc\GG'=\proc\GG''$ for $\GG'\ne\GG''$. We let $\ConfigModel(\cD)$ denote the set of all graphs $\GG$ consistent with a given neighborhood sequence $\cD$ (where $\GG$ is interpreted as we have just described). In particular,
	\beq\label{e:size.CM.D}
	|\ConfigModel(\cD)|
	= \prod_{\bm{t}} (n_{\bm{t}})!
	\,.
	\eeq
The notation $\ConfigModel$ stands for \textbf{configuration model}. A uniformly random element of $\ConfigModel(\cD)$ can be sampled by a generalization of the standard configuration model for graphs with given degree sequence --- this type of construction goes back to \cite{MR595929}, and we refer to \cite{MR1725006} for a survey. Generalized configuration models were analyzed in detail in \cite{MR3405616}, and the sampling procedure can be described as follows. Start with a collection of isolated vertices --- $n$ variables $V$ together with $m$ clauses $F$ --- labelled with total types according to $\cD$. Each vertex is then equipped with the appropriate number of ``half-edges,'' all labelled with edge total types. The total number of half-edges incident to each vertex corresponds to its degree in the processed graph $\GG$. Let $\delta V$ denote the variable-incident half-edges, and $\delta F$ the clause-incident half-edges. Then take a \textbf{uniformly random matching} between $\delta V$ and $\delta F$ that \bemph{respects edge total types}. (A full edge consists of one
variable-incident half-edge matched together with one clause-incident half-edge.) This procedure generates a uniformly random element of $\ConfigModel(\cD)$. This discussion is most relevant to the statement of Proposition~\ref{p:unif} below.\end{rmk}

The next three propositions, all proved in Section~\ref{s:processing}, give the key properties of the processed graph:

\begin{ppn}[proved in \S\ref{ss:preprocessing.probabilistic}: processing removes a small fraction of variables]\label{p:small.fraction.removed.in.processing}
Let $\GG\sim\P\equiv\P_{n,m}$ for $m$ such that $|m-n\alpha|\le n^{1/2}\log n$. Recall that the preprocessing algorithm (Definition~\ref{d:proc}) removes $\bsp'(A;\GG)$. We have
	\[
	\P\Bigg(
	\Big|B_{10R}\Big(\bsp'(A;\GG);\GG\Big)\Big|
		\ge\f{n}{\exp(2^{ ck }R)}
	\Bigg) = o_n(1)\,.
	\]
where $c$ is a positive absolute constant (depending only on the absolute constants $\KAPPA,\DELTACONST$).
 Moreover, the probability that 
any connected component of $\bsp'(A;\GG)$
contains a bicycle 
 is also $o_n(1)$.\footnote{We use ``bicycle'' to refer to any connected graph $G'=(V',F',E')$ with $|V'|+|F'|-|E'|=-1$.}
\end{ppn}

\begin{ppn}[proved in \S\ref{ss:positive.fraction.prob}: each surviving total type occurs linearly often] 
\label{p:posfrac}
Let $\GG\sim\P\equiv\P_{n,m}$ for $|m-n\alpha|\le n^{1/2}\log n$. Let $\proc\GG$ be the processed graph given by Definition~\ref{d:proc},
with neighborhood profile $\DD_{\proc\GG}$. Then
	\[
	\P\bigg(
	\min\bigg\{ 
	\hat\DD_{\proc\GG}(\bL)
	: \bL \textup{ is feasible}
	\bigg\} \ge c_1
	\,\bigg|\, \girth(\GG) > 8R
	\bigg) \ge 1-o_n(1)
	\]
where $c_1$ is a positive constant depending on $\KAPPA,\DELTACONST,k,R$, and the class of ``feasible'' $\bL$ is given by Definition~\ref{d:feasible.clause.type}.
\end{ppn}

\begin{ppn}[proved in \S\ref{ss:unif}: the processed graph is uniformly random given the neighborhood sequence]\label{p:unif}
Let $\GG\sim\P\equiv\P_{n,m}$ for any $m$. Let $\proc\GG$ be the processed graph given by Definition~\ref{d:proc}, with neighborhood sequence $\cD_{\proc\GG}$. For any $\cD$ such that
$\cD_{\proc\GG}=\cD$ with positive probability under $\P$, we have
	\[\P\Big(\proc\GG=H \,\Big|\,
		\cD_{\proc\GG}=\cD\Big)
	= \f{\Ind{H\in\ConfigModel(\cD)}}
		{\ConfigModel(\cD)}
	\]
for all $H\in\ConfigModel(\cD)$. Moreover, conditional on the neighborhood profile
$\DD_{\proc\GG}=\DD$, the law of the sequence $\cD_{\proc\GG}$ is uniformly random among all sequences $\cD$ with empirical counts $\DD$.
\end{ppn}

Write $\P_\cD$ for the uniform measure over the set $\ConfigModel(\cD)$ from Remark~\ref{r:processed.graph.types.CM}, and write $\E_\cD$ for expectation with respect to $\P_\cD$. 
Then Proposition~\ref{p:unif} tells us that 
	\[
	\P\Big(\proc\GG\in\cdot \,\Big|\,
		\cD_{\proc\GG}=\cD\Big)
	=\P_\cD(\cdot)\,.
	\]
This has the important consequence that the processed graph $\proc\GG$, conditional on $\cD$, can be sampled by the simple procedure described in Remark~\ref{r:processed.graph.types.CM}
--- this is essential to our analysis of the $\ksat$ model on the processed graph..
Next, write $\cD\sim\DD$ if $\cD$ has empirical counts given by $\DD$. Then 
Proposition~\ref{p:unif} also tells us that
	\[\P\Big(\proc\GG\in\cdot \,\Big|\,
		\DD_{\GG}=\DD\Big)
	= \P_{\DD}(\cdot)
	\equiv
	\sum_{\cD : \cD\sim\DD}
	\f{\P_\cD(\cdot)}{|\set{\cD' : \cD' \sim\DD}|}\,.
	\]
We will work sometimes with $\P_{\DD}$ and sometimes with $\P_\cD$, depending on convenience; they are equivalent modulo the ordering of the vertices.

\subsection{Proof outline for main theorem}\label{ss:contraction.overview} In this subsection we give a more detailed outline of the proof of the main result Theorem~\ref{t:main}.

\begin{dfn}[empirical measures of colors conditional on types]
\label{d:empirical.measures.of.colors}
 Take any $\ksat$ instance $\GG'$,
and let $\GG\equiv\proc\GG'$ be its processed version (Definition~\ref{d:proc}),
with total types as in Definition~\ref{d:total.type}.
Given a valid 
$\set{\RYGB}$-coloring $\usi$ of 
$\GG$ (Definition~\ref{d:color.model}), let $\pi$ be the \bemph{empirical measure of colors conditioned on edge type}:
for each $\sigma\in\set{\RYGB}$
and each edge total type $\bm{t}$,
	\beq\label{e:edge.msr.given.edge.type}
	\pi_{\bm{t}}(\sigma)
	\equiv
	\f{|\set{e\in E:\bm{t}_e=\bm{t}\text{ and }
	\sigma_e=\sigma}|}
	{|\set{e\in E:\bm{t}_e=\bm{t}}|}
	= \f{|\set{e\in E:\bm{t}_e=\bm{t}\text{ and }
	\sigma_e=\sigma}|}
	{n_{\bm{t}}}\,,
	\eeq
with $n_{\bm{t}}$ as in \eqref{e:empirical.counts.of.total.types}.
Further, let $\omega$ be the 
\bemph{empirical measure of colors conditioned on incident clause type}: for each $\sigma\in\set{\RYGB}$, each clause total type $\bL$, and each index $1\le j\le k(\bL)$ (where $k(\bL)\in\set{k-1,k}$ is the degree in $\GG=\proc\GG'$ of a clause of type $\bL$), let
	\beq\label{e:edge.msr.given.clause.type}
	\omega_{\bL,j}(\sigma)
	\equiv \f{|\set{(av)\in E:
		\bL_a=\bL,j(v;a)=j,
		\sigma_{av}=\sigma}|}{ m_{\bL}}
		\,,\eeq
with $m_{\bL}$ as in \eqref{e:empirical.counts.of.total.types}. Note that $\pi$ can be obtained as a marginal of $\omega$:
	\beq\label{e:pi.as.mgl.of.omega}
	n_{\bm{t}}
	\pi_{\bm{t}}(\sigma)
	= \sum_{j\in[k]}
	\sum_{\bL} 
	m_{\bL}
	\Ind{\bL(j)=\bm{t}}\omega_{\bL,j}(\sigma)
	\eeq
for all $\sigma\in\set{\RYGB}$, since both sides count the total number of edges of type $\bm{t}$ with color $\sigma$ in $\GG=\proc\GG'$.
\end{dfn}

The next two definitions are of essential importance, and are adapted from~\cite{MR3436404}: 

\begin{dfn}[judicious; adapted from {\cite{MR3436404}}] \label{d:judicious}
Let $\GG'$ be a $\ksat$ instance, $\GG\equiv\proc\GG'$ the processed graph, and $\usi$ any valid coloring of $\GG$ (Definition~\ref{d:color.model}).
We say that $\usi$ is \bemph{judicious}
if holds for all $\bL$ and all $j$ that
	\[
	\omega_{\bL,j} = \starpi_{\bL(j)} 
	\]
--- i.e., the empirical distribution of the edge color conditional on $(\bL,j)$ (the clause type and edge index) \bemph{depends only on $\bL(j)$} (the edge type itself, which carries less information), and moreover agrees (up to rounding) with the canonical edge marginal $\starpi$ of Definition~\ref{e:defn.canonical.edge.marginal}. For this definition, $\starpi_{av}$ should be based on the $r$-neighborhood of $v$ with respect to the processed graph $\GG$.
\end{dfn}

\begin{dfn}[separable; adapted from {\cite{MR3436404}}]
\label{d:separable}
Let $\GG'$ be a $\ksat$ instance, $\GG\equiv(V,F,E)
\equiv\proc\GG'$ the processed graph, and $\usi\in\set{\RYGB}^E$ any judicious coloring of $\GG$ (Definition~\ref{d:judicious}).
Let $\ux\equiv\ux(\usi)\in\set{\minus,\plus,\free}^V$ denote the frozen configuration on $\GG$ that corresponds to $\usi$ via \eqref{e:aux.model.bijections}. We say that $\usi$ is \bemph{separable} if there are not too many other judicious configurations $\usi'$ that are significantly correlated with $\usi$, that is, if
	\beq\label{e:middle.interval}
	\bigg|\bigg\{
	\textup{judicious }\usi' :
	\f{|v\in V: x(\usi)_v=x(\usi')_v|}{|V|}
	\notin \bigg[
	\f12\bigg(1-\f{k^4}{2^{k/2}}\bigg),
	\f12\bigg(1+\f{k^4}{2^{k/2}}\bigg)
	\bigg]\equiv I_0
	\bigg\}\bigg| \le 
	\exp\Big\{ (\log n)^5 \Big\}\,.
	\eeq
It is simpler and more convenient for our purposes that the correlation between $\usi$ and $\usi'$ is measured through the corresponding frozen configurations $\ux(\usi)$ and $\ux(\usi')$, rather than the colorings themselves.
\end{dfn}

\begin{dfn}[extendible; adapted from~\cite{MR3436404}]\label{d:extendible} Let $\GG=(V,F,E)$ be a (processed) $\ksat$ instance, and $\usi$ any valid coloring of $\GG$. Let $\ux\equiv\ux(\usi)$ be the frozen configuration corresponding to $\usi$ via \eqref{e:aux.model.bijections}. We say that $\usi$ is \bemph{extendible} if there exists an ``extension'' of $\ux\in\set{\minus,\plus,\free}^V$ to a satisfying assignment $\vec{\acute{x}}\in\set{\minus,\plus}^V$ of $\GG$, meaning $\acute{x}_v=x_v$ for all $v\in V$ where $x_v\in\set{\minus,\plus}$, and $\vec{\acute{x}}$ satisfies Definition~\ref{d:formal.dfn.sat}.
\end{dfn}

For the remainder of this section, $\GG'\sim\P\equiv\P^{n,\alpha}$ is a random $\ksat$ instance, and $\GG\equiv\proc\GG'$ is its processed version from Definition~\ref{d:proc}. We define the following quantities based on $\GG$:
	{\setlength{\jot}{0pt}\begin{align}\nonumber
	\ZZ
	\equiv \ZZ(\GG)
	& \equiv \textup{number of 
		judicious colorings $\usi$
		of $\GG$;}\\ \nonumber
	\sepZZ
	\equiv \sepZZ(\GG)
	&\equiv \textup{number of 
		judicious separable
		colorings $\usi$ of $\GG$;}\\
	\extZZ
	\equiv \extZZ(\GG)
	&\equiv \textup{number of 
		judicious extendible
		colorings $\usi$ of $\GG$.}
\label{e:def.ZZ}
\end{align}}%
Clearly, $\extZZ\le\ZZ$ and $\sepZZ\le\ZZ$. We emphasize that $\extZZ$ counts judicious extendible colorings of the processed graph $\GG$: by Definition~\ref{d:extendible}, 
such colorings extend to satisfying assignments of $\GG$. We are ultimately interested in whether they extend to satisfying assignments of the original instance $\GG'$; this discrepancy will be addressed below in the proof of Theorem~\ref{t:main} (at the end of this subsection).

Let $\GG'\sim\P$, and let $\DD=\DD_{\proc\GG'}$ be its processed neighborhood profile from
Definition~\ref{d:degseq}. We let $\LL(\cdot)\equiv\P(\DD_{\proc\GG'}\in \cdot)$ denote the law of $\DD$ itself, and let
	\beq\label{e:law.DD.girth}
	\LL_\textup{girth}(\cdot)
	\equiv\P\bigg(\DD_{\proc\GG'}\in \cdot
		\,\bigg|\, \girth(\GG') >8R \bigg)
	\eeq
denote the law conditional on $\girth(\GG') >8R$
Throughout this paper, when we say that an event holds ``with high probability over the random neighborhood profile $\DD$,'' we mean that the $\LL$-measure of the event is $1-o_n(1)$. We add the caveat ``conditional on $\girth(\GG') >8R$'' to mean that the 
$\LL_\textup{girth}$-measure of the event is $1-o_n(1)$. We have the following propositions
(assuming always that $k\ge k_0$
and $\alpha$ satisfies \eqref{e:alpha.regime}): 

\begin{ppn}[proved in \S\ref{ss:judicious.first.mmt}:
	first moment of judicious colorings matches $\onersb$ formula]
\label{p:first.moment.exponent}
Let $\Phi$ be the $\onersb$ free energy
from Proposition~\ref{p:phi}.
There exists $\ep_R$ which tends to zero as $R\to\infty$, such that
	\[\E_{\DD} \ZZ \ge
	\exp\bigg\{
	n\Big[\Phi(\alpha)-\ep_R\Big]\bigg\}
	\]
holds with high probability over the random $R$-neighborhood profile $\DD$.
\end{ppn}

\begin{ppn}[proved in \S\ref{ss:extendibility}:
	first moment of judicious colorings
	dominated by extendible colorings]
\label{p:ext}
For the random variables $\extZZ\le\ZZ$ as in \eqref{e:def.ZZ}, we have
	\[\E_{\DD}(\extZZ)
	=\bigg\{ 1-o_n(1)\bigg\} \E_{\DD}\ZZ\]
with high probability over the random $R$-neighborhood profile $\DD$.
\end{ppn}

\begin{ppn}[proved in \S\ref{ss:separability}:
	first moment of judicious colorings
	dominated by separable colorings]
\label{p:sep}
For the random variables $\sepZZ\le\ZZ$ as in \eqref{e:def.ZZ}, we have
	\[\E_{\DD}(\sepZZ)
	=\bigg\{ 1-o_n(1)\bigg\} \E_{\DD}\ZZ\]
with high probability over the random $R$-neighborhood profile $\DD$.
\end{ppn}

The most difficult result of this paper is a second moment estimate. To state it, let us decompose
	\beq\label{e:second.mmt.decompose.by.z}
	\ZZ^2
	\equiv \sum_z \ZZ^2[z]\eeq
where $\ZZ^2[z]$ denotes the contribution
from pairs $(\usi^1,\usi^2)$
whose corresponding frozen configurations
$\ux(\usi^i)$ agree on exactly $z$ fraction of the variables in $\GG$. For any subset $I\subseteq[0,1]$, let $\ZZ^2[I]$ denote the sum of $\ZZ^2[z]$ over $z\in I$. The central part of the paper is concerned with the following estimate:

\begin{ppn}[proved in \S\ref{ss:hess}: \bemph{main technical result}]\label{p:second.moment.judicious}
Recall from \eqref{e:middle.interval}
the definition of the interval $I_0\subset[0,1]$.
There is a constant $C\equiv C(k,R)$ such that
	\[\E_{\DD}\Big( \ZZ^2[I_0]\Big)
	\le C \Big(\E_{\DD}\ZZ\Big)^2\]
with high probability over the random
$R$-neighborhood profile $\DD$, conditional on $\girth(\GG')>8R$.
\end{ppn}

In \S\ref{ss:first.moment.preliminaries}
and \S\ref{ss:second.moment.overview} we give further discussion on ideas of the proofs of Propositions~\ref{p:first.moment.exponent}~and~\ref{p:second.moment.judicious}.
For now we turn to explaining how the above propositions imply the main result. We rely on the following well-known (and elementary) bound: if $Y$ is any non-negative random variable with finite second moment, then for any $0<\delta<1$,
	\[
	\E\bigg[Y ; Y\ge\delta\,\E Y\bigg]
	=\E Y - \E\bigg[Y ; Y<\delta\,\E Y\bigg]
	\ge (1-\delta)\E Y.
	\]
On the other hand, by the Cauchy--Schwarz inequality,
	\[
	\E\bigg[Y ; Y\ge\delta\,\E Y\bigg]
	\le\bigg\{
	 \E(Y^2) \P\Big(Y\ge\delta\,\E Y\Big)
	 \bigg\}^{1/2}\,.
	\]
Combining the bounds and rearranging gives
	\beq\label{e:second.mmt.method.with.const}
	\P\Big(Y\ge\delta\,\E Y\Big)
	\ge (1-\delta)^2 \f{(\E Y)^2 }{ \E(Y^2)}\,.
	\eeq
Thus, any estimate of the form $\E(Y^2)\le O((\E Y)^2)$ gives a lower bound $\P(Y\ge \delta\,\E Y) \ge \Omega(1)$.

\begin{proof}[Proof of Theorem~\ref{t:main}]
Propositions~\ref{p:fp}~and~\ref{p:phi} together show that for random $\ksat$ with $k\ge k_0$, the $\onersb$ free energy $\Phi(\alpha)$ is well-defined, with a unique root $\arsb$ in the interval \eqref{e:alpha.regime}. It follows from Proposition~\ref{p:ubd} that $\arsb$ upper bounds the satisfiability regime, so it remains to show the lower bound. To this end, let $\alpha<\arsb$
(still within the regime \eqref{e:alpha.regime}),
so that $\Phi(\alpha)>0$. We divide the rest of the argument into two parts:\smallskip

\noindent\bemph{Step 1. Lower bound on separable colorings.}
Analogously to \eqref{e:second.mmt.decompose.by.z}, 
decompose 
	\[(\sepZZ)^2=\sum_z (\sepZZ)^2[z]\,.\]
Let $(\sepZZ)^2[I]$
denote the sum of $(\sepZZ)^2[z]$ over $z\in I\subseteq[0,1]$.
By the separability condition (Definition~\ref{d:separable}),
	\[(\sepZZ)^2\Big[[0,1]\setminus
	 I_0\Big]
	 \le e^{o(n)} \sepZZ
	 \le e^{o(n)} \ZZ\]
almost surely. Recall the definition of $\LL_\textup{girth}$ from \eqref{e:law.DD.girth}.
Proposition~\ref{p:second.moment.judicious}
gives the bound
	\[\LL_\textup{girth}\Bigg(
	\E_{\DD}\Big[(\sepZZ)^2[I_0]\Big]
	\le\E_{\DD}\Big[\ZZ^2[I_0]\Big]
	\le C\Big(\E_{\DD}\ZZ\Big)^2
	\Bigg) = 1-o_n(1)\]
for $C=C(k,R)$. Combining the above bounds gives
	\beq\label{e:t:second.moment}
	\LL_\textup{girth}\Bigg( \E_{\DD}\Big((\sepZZ)^2\Big)
	\le C \Big(\E_{\DD}\ZZ\Big)^2
	+ e^{o(n)}\,\E_{\DD}\ZZ
	\Bigg) = 1-o_n(1)\,.
	\eeq
On the other hand, it follows by Propositions~\ref{p:first.moment.exponent}~and~\ref{p:sep} that
	\beq\label{e:sep.compared.with.ZZ.first.mmt}
	\LL
	\Bigg( \E_{\DD}\sepZZ \ge \f{\E_{\DD}\ZZ}{2}
		\ge 
		\f12 \exp\bigg\{
	n\Big[\Phi(\alpha)-\ep_R\Big]\bigg\}
	\Bigg) = 1-o_n(1)\,.
	\eeq
Since we chose $\alpha<\arsb$, we know that $\Phi(\alpha)$ is positive, and it follows for large enough $R$ (depending on $k,\alpha$) that $\E_{\DD}\ZZ$ is exponentially large in $n$. It follows that
\eqref{e:t:second.moment} 
and \eqref{e:sep.compared.with.ZZ.first.mmt}
combine to yield
	\[
	\LL_\textup{girth}\Bigg( \E_{\DD}\Big((\sepZZ)^2\Big)
	\le 5C \Big(\E_{\DD}(\sepZZ)\Big)^2
	\Bigg) = 1-o_n(1)\,.
	\]
Applying \eqref{e:second.mmt.method.with.const},
we see that there exists a positive constant $\delta=\delta(k,R)$ such that 
	\beq\label{e:sep.second.mmt.final.lbd}
	\LL_\textup{girth}\Bigg(
	\P_{\DD}\bigg[
	\sepZZ \ge \delta \, \E_{\DD}(\sepZZ)\bigg]
	\ge \delta
	\Bigg) = 1-o_n(1)\,.
	\eeq

\noindent\bemph{Step 2. Lower bound on extendible colorings.} It follows from Proposition~\ref{p:ext} 
and Markov's inequality that,
with $\delta=\delta(k,R)$ as above, we have
	\[
	\P_{\DD} \bigg(
	\ZZ-\extZZ \ge \delta^2 \, \E_{\DD}\ZZ
	\bigg)
	\le \f{\E_{\DD}(\ZZ-\extZZ)}{\delta^2\,\E_{\DD}\ZZ}
	= o_n(1)
	\]
with high probability over the random neighborhood profile $\DD$. Combining with \eqref{e:sep.compared.with.ZZ.first.mmt}
and \eqref{e:sep.second.mmt.final.lbd} gives
	\[\LL_\textup{girth}\Bigg(
	\P_{\DD}\bigg[
	\extZZ
	\ge \ZZ - \delta^2 \, \E_{\DD}\ZZ
	\ge \sepZZ - \delta^2 \, \E_{\DD}\ZZ
	\ge \f{\delta(1-2\delta)}{2}
	\exp\bigg\{
	n\Big[\Phi(\alpha)-\ep_R\Big]\bigg\}
	>0
	\bigg] \ge \f{\delta}{2}\Bigg) = 1-o_n(1)\,.\]
Recall the definition of $\LL_\textup{girth}$ from \eqref{e:law.DD.girth}: since the graph $\GG'\sim\P$ has $\girth(\GG')>8R$ with asymptotically positive probability, we deduce from the above that
	\beq\label{e:ext.ZZ.cond.lbd}
	\LL\Bigg(
	\P_{\DD}\Big[\extZZ>0\Big]
	\ge \f{\delta}{2}\Bigg)
	\ge \delta'= \delta'(k,R)\,.
	\eeq
As long as no connected component of $\GG'\setminus\proc\GG'$ contains a bicycle,
any satisfying assignment of 
$\GG=\proc\GG'$ extends to a satisfying assignment of $\GG'$. It follows by combinding with
Definition~\ref{d:extendible} that
	\[
	\mathbf{1}\bigg\{
	\begin{array}{c}
	\textup{no connected component of}\\
	\textup{$\GG'\setminus\proc\GG'$ contains a bicycle}
	\end{array}
	\bigg\}
	\Ind{\extZZ>0}
	\le\Ind{\textup{$\GG'$ is satisfiable}}\,.
	\]
From Proposition~\ref{p:small.fraction.removed.in.processing} that, with high probability, no connected component of
$\GG'\setminus\proc\GG'$ contains a bicycle.
Combining with \eqref{e:ext.ZZ.cond.lbd}
gives
	\[
	\P^{n,\alpha}\Big(\GG'
		\textup{ is satisfiable}\Big)
	\ge
	\E\Bigg(
	\mathbf{1}\bigg\{
	\begin{array}{c}
	\textup{no connected component of}\\
	\textup{$\GG'\setminus\proc\GG'$ contains a bicycle}
	\end{array}
	\bigg\}
	\cdot
	\P_{\DD}\Big[\extZZ>0\Big]\Bigg)
	\ge \delta''
	\]
for any $\alpha<\arsb$ (where $\delta''$ depends on $k$ and $R$; and $R$ depends on $k$ and $\alpha$).
Combining with Friedgut's theorem \cite{Friedgut:99} gives $\asat\ge \alpha$. 
The result follows by taking
$\alpha\uparrow\arsb$
and $R\uparrow\infty$.
\end{proof}

\begin{rmk}\label{r:where.proc.results.used} 
Note that Propositions \ref{p:small.fraction.removed.in.processing}--\ref{p:unif} (stated in \S\ref{ss:introprep}, proved in Section~\ref{s:processing}) are not directly referenced in the above proof of Theorem~\ref{t:main}. Instead, they are used indirectly via the other propositions. In particular, for the first moment (Proposition~\ref{p:first.moment.exponent}), the uniformity result (Proposition~\ref{p:unif}) provides a simple combinatorial formula for $\E_{\DD}\ZZ$. We then need the result that only $o_R(1)$ fraction of the variables are removed (Proposition~\ref{p:small.fraction.removed.in.processing}) to relate the combinatorial formula to the $\onersb$ free energy \eqref{e:phi.alpha}. The second moment (Proposition~\ref{p:second.moment.judicious}) similarly relies on Proposition~\ref{p:unif} for the combinatorial formula, but additionally requires the result that each type occurs linearly many times
(Proposition~\ref{p:posfrac}) to estimate the combinatorial formula up to a \emph{constant} multiplicative error.
\end{rmk}

This concludes our overview of the proof of Theorem~\ref{t:main}. The remainder of this section is organized as follows:
\begin{enumerate}[--]
\item In \S\ref{ss:first.moment.preliminaries} we describe the basic ideas in the proof of of Proposition~\ref{p:first.moment.exponent}. 
\item In \S\ref{ss:second.moment.overview} we elaborate on some of the principles in the proof of Proposition~\ref{p:second.moment.judicious}. 
\item In \S\ref{ss:coherence.weights} we prove our main claims regarding coherent clauses, thereby concluding the current section.
\end{enumerate}
At the end of this section (page \pageref{proofoverview}), an outline of the remainder of the paper is provided.

\subsection{First moment and the Bethe formula for colorings}\label{ss:first.moment.preliminaries}

In this subsection we give some basic calculations
for the proofs of 
Propositions~\ref{p:first.moment.exponent}~and~\ref{p:second.moment.judicious}.
Again, let $\GG'$ be a $\ksat$ instance, and $\GG\equiv\proc\GG'$ the processed graph.
Recall from Definition~\ref{d:empirical.measures.of.colors} that if $\usi$ is any valid coloring of $\GG$, then it defines two conditional empirical measures --- $\pi$ as in \eqref{e:edge.msr.given.edge.type}, and $\omega$ as in \eqref{e:edge.msr.given.clause.type}. We now further define:

\begin{dfn}[empirical measures of vertex-incident colorings]
\label{d:pi.omega.nu.notation}
As before, let $\GG'$ be a $\ksat$ instance, 
and denote its processed version
$\GG\equiv\proc\GG'$. 
Given a valid coloring $\usi$ of $\GG$, we define $\dbh$ to be the \bemph{empirical measure of variable colorings conditional on type},
	\beq\label{e:variable.color.cond.on.type}
	\dbh_{\bT}(\usi_{\delta v})
	\equiv
		\f1{n_{\bT}}
		\Bigg|\bigg\{w\in V: \textup{$\bT_w=\bT$
			and $\usi_{\delta w}
				=\usi_{\delta v}$}
				\bigg\}\Bigg|\,,
	\eeq
with $n_{\bT}$ as in \eqref{e:empirical.counts.of.total.types}.
We define $\hbh$ \bemph{empirical measure of clause colorings conditional on type},
	\beq\label{e:clause.color.cond.on.type}
	\hbh_{\bL}(\usi_{\delta a})
	\equiv 
	\f1{m_{\bL}}
	\Bigg|
	\bigg\{b\in F:
		\textup{$\bL_b=\bL$
		and $\usi_{\delta b}=\usi_{\delta a}$}
		\bigg\}
	\Bigg|\,,\eeq
with $m_{\bL}$ as in \eqref{e:empirical.counts.of.total.types}.
As a shorthand, we call $\nu\equiv(\dbh,\hbh)$ the \bemph{empirical measure (of vertex colorings)} associated to $\usi$. Recall from \eqref{e:pi.as.mgl.of.omega} that $\pi$ can be obtained as a marginal of $\omega$. Now note that $\pi$ can also be obtained as a marginal of $\dbh$. Similarly, $\omega$ can be obtained as a marginal of $\hbh$. We say that $\nu$ is \bemph{judicious} if its marginal $\omega$ is judicious in the sense of Definition~\ref{d:judicious}. For $\ZZ$ as in \eqref{e:def.ZZ} and any judicious empirical measure $\bh\equiv(\dbh,\hbh)$, we let $\ZZ[\bh]$ denote the contribution to $\ZZ$
from colorings with empirical measure $\bh$.
\end{dfn}

\begin{rmk}\label{r:conflate}
We sometimes abuse terminology slightly by conflating edges with edge types, and vertices with vertex types. For example, when we say that we fix a type-$\bL$ clause $a$ and consider the distribution of colorings $\usi_{\delta a}$, we are referring to the empirical measure of these colorings among all type-$\bL$ clauses in the graph. We will use the notations $\hat{\nu}_{\bL}$ and $\hat{\nu}_a$ interchangeably for this measure.
\end{rmk}

Let $\GG'\sim\P_{n,m}$ be the random $\ksat$ instance, and let $\GG\equiv(V,F,E)\equiv\proc\GG'$ be the processed graph, with neighborhood profile $\DD$ as in Definition~\ref{d:degseq}. For $\bh=(\dbh,\hbh)$ as in Definition~\ref{d:pi.omega.nu.notation}, define
	\beq\label{e:rate.function.given.gen.degseq}
	\bm{\Phi}_{\DD}(\bh)
	=\f1n \Bigg\{
	|V| \, \E_{\dot{\DD}}[ \Ent(\dbh_{\bT}) ]
	+|F| \, \E_{\hat{\DD}}[ \Ent(\hbh_{\bL}) ]
	-|E| \, \E_{\bar{\DD}}
		[ \Ent(\pi_{\bm{t}}) ]
		\Bigg\}
	\eeq
(where $\Ent$ denotes the usual entropy function, $\Ent(p)=-\sum_x p_x \log p_x$). In the above expression, the expectations on the right-hand side refer to sampling total types according to $\DD$: for instance,
	\beq\label{e:n.bt.as.related.to.DD}
	|E| \, \E_{\bar{\DD}}
		[ \Ent(\pi_{\bm{t}}) ]
	\equiv
	|E|
	\sum_{\bt} \bar{\DD}(\bt)\Ent(\pi_{\bm{t}})
	= - \sum_{\bt} n_{\bt}
		\sum_{\sigma}
		\pi_{\bm{t}}(\sigma)\log\pi_{\bm{t}}(\sigma)
	\,,
	\eeq
with $n_{\bt}$ as in \eqref{e:empirical.counts.of.total.types}. For given $R$ there are only finitely many total types possible,
so $\nu$ lies in a simplex of bounded dimension. Let $s_{\bt}=|\supp\starpi_{\bt}|$.
Then let $s_{\bT}$ denote the number of colorings $\usi_{\delta v}$ that can appear on a variable $v$ of type $\bT$.
Define likewise $s_{\bL}$ for clause types $\bL$, and let
	\beq\label{e:stirling.correction.number}
	\bm{s} \equiv \bm{s}(\DD)
	\equiv \sum_{\bT} (s_{\bT}-1)
		+ \sum_{\bL} (s_{\bL}-1)
		- \sum_{\bt} (s_{\bt}-1)\,.
	\eeq
Next, recalling the statement of Proposition~\ref{p:posfrac}, we will say that \bemph{``$\DD$ is bounded away from zero''} to mean that
	\beq\label{e:pos.frac}
	\min\bigg\{ 
	\hat\DD(\bL)
	: \bL \textup{ is feasible}
	\bigg\} \ge c_1\,,\eeq
where ``feasible'' is specified by Definition~\ref{d:feasible.clause.type}. We then have the following:

\begin{lem}\label{l:combinatorial.mmt.simplification}
Let $\GG'\sim\P_{n,m}$ be the random $\ksat$ instance,
and let $\GG=\proc\GG'$ be its processed version,
with neighborhood profile $\DD$ as in Definition~\ref{d:degseq}.
Let $\P_{\DD}$ be the uniform measure over $\ConfigModel(\DD)$ (Remark~\ref{r:processed.graph.types.CM}), and $\E_{\DD}$ the expectation with respect to $\P_{\DD}$. For any $\DD$ and any judicious $\nu$, we have,
with $\bm{\Phi}_{\DD}(\nu)$ as in \eqref{e:rate.function.given.gen.degseq},
	\beq\label{e:stirling.with.poly.error}
	\E_{\DD}\ZZ[\bh]
	= \f1{n^{O(1)} }
	\exp\bigg\{ n\bm{\Phi}_{\DD}(\bh)
		\bigg\}\,.
	\eeq
If $\DD$ is bounded away from zero 
and $\nu$ lies in the interior of its simplex, then
	\beq\label{e:stirling.with.poly.error}
	\E_{\DD} \ZZ[\bh]
	\asymp_R
	\f1{ n^{\bm{s}/2}}
	\exp\bigg\{ n\bm{\Phi}_{\DD}(\nu)\bigg\}
	\eeq
for $\bm{s}\equiv\bm{s}(\DD)$ as in \eqref{e:stirling.correction.number}.

\begin{proof}
It follows from the description of the configuration model (Remark~\ref{r:processed.graph.types.CM}) that
	\beq\label{e:config.model.first.moment.given.omega}
	\E_{\DD}\ZZ[\bh]
	=\underbrace{
	\Bigg\{
	\prod_{\bT} 
	\binom{n_{\bT}}{n_{\bT}\dbh_{\bT}}
	\prod_{\bL}
	\binom{m_{\bL}}{m_{\bL}\hbh_{\bL}}
	\Bigg\}
	}_{\substack{
	\text{number of colorings}\\
	\text{prior to matching}
	}}
	\underbrace{\Bigg\{
	\prod_{\bm{t}} \binom{n_{\bm{t}}}
		{ n_{\bm{t}} \pi_{\bm{t}} }
		\Bigg\}^{-1}
		}_{\substack{\text{probability of matching}\\
			\text{to respect colorings}}}\,.
	\eeq
The first estimate \eqref{e:stirling.with.poly.error}
then follows by Stirling's formula, ignoring polynomial corrections.
The second estimate \eqref{e:stirling.with.poly.error} follows by taking the polynomial corrections into account.
\end{proof}
\end{lem}

Returning to the form of $\bm{\Phi}_{\DD}(\nu)$ in \eqref{e:rate.function.given.gen.degseq}, note that if $\omega$ is fixed,
then $\pi$ is determined (from Definition~\ref{d:empirical.measures.of.colors}), and we see that
$\bm{\Phi}_{\DD}(\nu)$ is a \bemph{strictly concave} function of $\bh$ for fixed $\omega$. As a result, there is a unique maximizer: we let
	\beq\label{e:nu.opt}
	\bm{\Psi}_{\DD}(\omega)
	\equiv\bm{\Phi}_{\DD}(\optnu[\omega])\,,\quad
	\optnu[\omega]
	\equiv \argmax_\nu
	\bigg\{
	\bm{\Phi}_{\DD}(\nu):
	\textup{$\nu$ is consistent with $\omega$}
	\bigg\}\,.\eeq
In fact, as we discuss below, $\optnu[\omega]$ takes a rather explicit form \eqref{e:lagrange.clause.tuple}, which eventually allows us to relate $\bm{\Psi}_{\DD}(\omega)$ to the free energy $\Phi(\alpha)$ from \eqref{e:phi.alpha}. We defer this for the moment, and proceed with the calculation of $\E_{\DD}\ZZ$ in terms of $\bm{\Psi}_{\DD}(\omstar)$. Let $x_{\bT}\in\set{0,1,2}$ count the number of frozen spins ($\plus$ or $\minus$) that can appear on a variable $v$ of type $\bT$. (Thus $x_{\bT}=0$ indicates that variables of type $\bT$ must always be free, while $x_{\bT}=2$ indicates that variables of type $\bT$ can take any spin in $\set{\minus,\plus,\free}$.) Then define
	\beq\label{e:wp.of.R}
	\wp\equiv\wp(\DD)
	\equiv
	\sum_{\bL} \Bigg\{
	\sum_j (s_{\bL(j)}-1)\Bigg\}
	- \sum_{\bT} x_{\bT}\Big(|\delta v|-1\Big)
	\,.
	\eeq
In the first term on the right-hand side of \eqref{e:wp.of.R}, the outer sum goes over all clause total types $\bL$, while the inner sum goes over $1\le j\le k(\bL)$ where $k(\bL)$ is the degree in $\GG$ of a clause of type $\bL$.

\begin{cor}\label{c:first.moment.exponent}
Let $\GG'\sim\P_{n,m}$ be the random $\ksat$ instance, and let $\GG=\proc\GG'$ be its processed version, with neighborhood profile $\DD$ as in Definition~\ref{d:degseq}. If $\DD$ is bounded away from zero in the sense of \eqref{e:pos.frac}, then
	\beq\label{e:first.mmt.poly}
	\E_{\DD}\ZZ 
	\asymp_R
	\f{\exp\{
		n\bm{\Psi}_{\DD}(\omstar) \}
	}{ n^{\wp/2} }
	\eeq
for $\bm{\Psi}_{\DD}$ as in \eqref{e:nu.opt}, $\wp=\wp(\DD)$ as in \eqref{e:wp.of.R}, and $\omstar$ defined by $(\omstar)_{\bL,j}=\starpi_{\bL(j)}$.

\begin{proof}
Again recall from \eqref{e:def.ZZ} that $\ZZ$ counts judicious colorings (Definition~\ref{d:judicious}) $\usi$ of $\GG$. For any such $\usi$, the empirical measure $\omega$ (Definition~\ref{d:empirical.measures.of.colors}) is completely fixed by the judicious condition: it must agree (up to rounding) with the measure $\omstar$ defined by $(\omstar)_{\bL,j} = \starpi_{\bL(j)}$.\footnote{Recall from Definition~\ref{d:total.type} that $\bL(j)$ encodes the neighborhoods in both the initial graph $\GG'$ and the final graph $\GG=\proc\GG'$. However, $\starpi_{\bL(j)}$ should be given by Definition~\ref{d:canonical} applied to the $r$-neighborhood of the variable \bemph{in the final graph} $\GG$.} Let us abbreviate $\nu\sim\omstar$ if $\nu$ is consistent with $\omstar$. We will show in Lemma~\ref{l:dimension.count}
(this lemma is deferred to \S\ref{ss:coherence.weights}) that
	\beq\label{e:dim.nu.given.omega}
	d_1(\DD)
	\equiv \dim\bigg\{\nu : \nu\sim\omstar\bigg\}
	= \bm{s}(\DD)-\wp(\DD)\,,
	\eeq
for $\bm{s}(\DD)$ as in \eqref{e:stirling.correction.number} and $\wp(\DD)$ as in \eqref{e:wp.of.R}. It follows from Lemma~\ref{l:combinatorial.mmt.simplification} that
	\beq\label{e:prep.for.gaussian.sum}\E_{\DD}\ZZ
	= \sum_{\nu: \nu\sim\omstar}
	\E_{\DD}\ZZ[\bh]
	\asymp_R
	\sum_{\nu: \nu\sim\omstar}
	\f{\exp\{ n\bm{\Phi}_{\DD}(\nu) \}}
		{ n^{\bm{s}/2} }
	\eeq
Now recall the general fact that for any fixed dimension $d$, in the limit $n\to\infty$ we have
	\beq\label{e:gaussian.sum}
	\sum_{x \in (\mathbb{Z}/n)^d}
		\f1{\exp(n\|x\|^2)}
	= \sum_{x \in (\mathbb{Z}/n^{1/2})^d}
		\f1{\exp( \|x\|^2)}
	\asymp_d n^{d/2}\,.\eeq
Since we already noted above that $\bm{\Phi}_{\DD}(\nu)$ is strictly concave in $\bh$ for fixed $\omega$, the claimed result follows by
applying \eqref{e:gaussian.sum}
to \eqref{e:prep.for.gaussian.sum},
with dimension $d$ given by $\bm{s}(\DD)-\wp(\DD)$ as in \eqref{e:dim.nu.given.omega}.
\end{proof}
\end{cor}

Let us now return to the explicit optimization problem \eqref{e:nu.opt}. Suppose $\omega$ is given, so that $\pi$ is determined by the relation \eqref{e:pi.as.mgl.of.omega}. The optimal $\optdbh[\omega]$ depends on $\omega$ only through $\pi$: it is given by finding, for each variable type $\bT$, the measure $\dbh_{\bT}$ (on valid colorings of such a variable) that maximizes entropy and is consistent with marginals $\pi$. As in Remark~\ref{r:conflate}, let us abuse notation and write $\dbh_{\bT}\equiv\dbh_v$ where $v$ is a variable of type $\bT$,
and $\pi_{\bt}\equiv\pi_e$
where $e$ is an edge of type $\bt$. Then 
	\[
	(\optdbh[\omega])_v
	=\argmax_\nu\bigg\{
	\Ent(\dbh_v) : 
	\pi_e(\sigma)
	= \sum_{\usi_{\delta v}}
	\Ind{\sigma_e=\sigma} \dbh_v(\usi_{\delta v})
	\textup{ for all $e\in\delta v$, 
		$\sigma\in\set{\RYGB}$}
	\bigg\}\,.
	\]
The associated Lagrangian is given by
	\[-
	\sum_{\usi_{\delta v}}
	\dbh_v(\usi_{\delta v})
	\log\dbh_v(\usi_{\delta v})
	+ \sum_{e\in\delta v}\sum_\sigma
		\lambda_e(\sigma)
		\bigg[
		\sum_{\usi_{\delta v}} \Ind{\sigma_e=\sigma} \dbh_v(\usi_{\delta v})
		- \pi_e(\sigma)
		\bigg]\,,
	\]
where the first sum goes over $\usi_{\delta v}$
for which $\varphi_v(\usi_{\delta v})$,
as defined by 
\eqref{e:color.model.variable.factor},
equals one. It follows that
	\beq\label{e:lagrange.clause.tuple}
	(\optdbh[\omega])_v(\usi_{\delta v})
	\cong
	\varphi_v(\usi_{\delta v})
	\prod_{e\in\delta v} q_e(\sigma_e)
	\eeq
where $\cong$ indicates the overall normalization, and the $q_e$ are probability measures on $\set{\RYGB}$ such that $\optdbh[\omega]$ satisfies the constraint of having marginals consistent with $\pi$. 
Note the clear resemblance between
\eqref{e:var.tuple.measure.weighted}
and \eqref{e:lagrange.clause.tuple}.
The above is for general $\omega$; in the main special case of interest we have:

\begin{lem}\label{l:nu.star.as.optimizer.for.starpi}
For any $\omega$ with marginals $\starpi$,
$\optdbh[\omega]$ is given by \eqref{e:lagrange.clause.tuple}
with $q_e=\hqstar_e$ for all $e$, that is to say, by
	\beq\label{e:opt.variable.tuple.measure.star}
	\dbhstar_v(\usi_{\delta v})
	\equiv \f1{\dbz_v}
	\varphi_v(\usi_{\delta v})
	 \prod_{e\in\delta v} 
		\hqstar_e(\sigma_e)\eeq
where $\dbz_v$ is the normalizing constant. 

\begin{proof}
It suffices to verify that
	\begin{align}\nonumber
	\dbhstar_v(\sigma_{av}=\sigma)
	&\cong
	\hqstar_{av}(\sigma_{av})
	\sum_{\usi_{\delta v}}
	\Ind{\sigma_{av}=\sigma}
	\varphi_v(\usi_{\delta v})
	\prod_{b \in \pd v\setminus a}
		\hqstar_{bv}(\sigma_{bv})\\
	&\cong\hqstar_{av}(\sigma_{av})
	\bigg\{
	\Big(\BP_{va}[\hqstar]\Big)(\sigma_{av})
	\bigg\}
	=\hqstar_{av}(\sigma_{av})
	\dqstar_{va}(\sigma_{av})
	\cong \starpi_{av}(\sigma)\,,
	\label{e:marginal.of.nu.BP}
	\end{align}
where $\BP$ is the belief propagation mapping for the color model (the unweighted version of \eqref{e:bp}),
and the last two identities follow from the fact that $\starpi_{av}$, $\dqstar_{va}$, and $\hqstar_{av}$ are all based on the \bemph{same} tree, namely,
the $r$-neighborhood of $v$ in $\GG$
(see also Remark~\ref{r:first.rmk.coherence}).
This verifies that the marginals of $\dbhstar$ are indeed consistent with $\starpi$. Since $\dbhstar$ takes the form \eqref{e:lagrange.clause.tuple} given by the Lagrangian calculation, it follows that $\dbhstar=\optdbh[\omega]$.
\end{proof}
\end{lem}

In the optimization \eqref{e:nu.opt}, if $\omega$ is given, then the optimal $\opthbh[\omega]$ 
solves a similar problem of maximizing entropy
subject to marginals $\omega$.
 It is important to note, however, that 
the analogue of Lemma~\ref{l:nu.star.as.optimizer.for.starpi} 
does \bemph{not} hold for the clause measure:
for $(\omstar)_{\bL,j}=\starpi_{\bL(j)}$,
it is \bemph{not} necessarily the case
that $\opthbh[\omstar]$ is given by
(cf.\ \eqref{e:clause.tuple.measure.weighted})
	\beq\label{e:opt.clause.tuple.measure.star}
	\hbhstar_a(\usi_{\delta a})
	=
	\f1{\hbz_a}
	\hat{\varphi}_a(\usi_{\delta a})
		\prod_{v\in\pd a}
		\dqstar_{va}(\sigma_{av})\,.
	\eeq
This is simply because
$\hbhstar$ need not be consistent with $\omstar$:
a similar calculation as \eqref{e:marginal.of.nu.BP} gives
	\[
	\hbhstar_a(\sigma_{av}=\sigma)
	\cong
	\dqstar_{va}(\sigma_{av})\bigg\{
	\Big(\BP_{av}[\dqstar]\Big)(\sigma_{av})
	\bigg\}\,.
	\]
This generally does not match $\starpi_{av}(\sigma_{av})$ since, as noted in Remark~\ref{r:first.rmk.coherence}, $\BP_{av}[\dqstar]$ need not equal $\hqstar_{av}$.
A significant part of the proof of Proposition~\ref{p:first.moment.exponent}
is concerned with this discrepancy between
the explicit measure $\hbhstar$
of \eqref{e:opt.clause.tuple.measure.star}
and the optimizer $\opthbh[\omstar]$.
This calculation is deferred to \S\ref{ss:judicious.first.mmt}.

Given the above discussion, it is now easy to \emph{guess} what value $n^{-1}\log\E_{\DD}\ZZ$ should concentrate around, in the limit $n\to\infty$ followed by $R\to\infty$. Since the fraction of variables removed by preprocessing is $o_R(1)$ (Proposition~\ref{p:small.fraction.removed.in.processing}), the measure $\DD$ should concentrate around the Galton--Watson measure that is the local weak limit of the original $\ksat$ graph $\GG'$ (for the details see Definition~\ref{d:uPGW}). The discrepancy between $\hbhstar$ and $\opthbh[\omstar]$ should also go away as $R\to\infty$. The canonical messages $\dqstar$ and $\hqstar$ should converge, in a distributional sense, to the limiting measures on the Galton--Watson tree. We formalize this as follows:

\begin{dfn}[random messages for the color model]
\label{d:bethe.laws.of.color.messages}
Let $\mu$ be the measure given by Proposition~\ref{p:fp}, and let $\ueta$ be an array of i.i.d.\ samples from $\mu$. Let $\ud \equiv (d^\plus,d^\minus)$ be two independent samples from the $\Pois(\alpha k/2)$ distribution. Define $\Pi^\PM\equiv\Pi^\PM(\vec d,\vec\eta)$ as in \eqref{e:Pi.PM}, and define
(cf.\ \eqref{e:intro.dist.recurs})
	\[
	\bmeta(\plus)
	= \f{(1-\Pi^\plus)\Pi^\minus}
		{\Pi^\plus+\Pi^\minus-\Pi^\plus\Pi^\minus},\quad
	\bmeta(\minus)
	= \f{(1-\Pi^\minus)\Pi^\plus}
		{\Pi^\plus+\Pi^\minus-\Pi^\plus\Pi^\minus},\quad
	\bmeta(\free)
	= \f{\Pi^\plus\Pi^\minus}
		{\Pi^\plus+\Pi^\minus-\Pi^\plus\Pi^\minus}\,,
	\]
so that $\bmeta$ is a (random) probability measure over $\set{\minus,\plus,\free}$. Now let $\bmeta_j$ be i.i.d.\ copies of $\bmeta$, and define (cf.\ \eqref{e:first.defn.of.bhu})
	\beq\label{e:second.defn.of.bhu}
	\bm{\hat{u}}
	\equiv 
	\Big(\bm{\hat{u}}(\plus),\bm{\hat{u}}(\free)\Big)
		= 
		\bigg(
		\prod_{j=1}^{k-1} \bmeta_j(\minus),
		1-\prod_{j=1}^{k-1} \bmeta_j(\minus)
		\bigg)\,.\eeq
We substitute the random measures $\bmeta$ and $\bm{\hat{u}}$
into \eqref{e:color.recursions.eta}
to define random messages for the color model:
	\begin{align}
	\label{e:color.recursions.eta.again}
	\dq(\bmeta)
	&=(\dq(\red),
		\dq(\yel),
		\dq(\grn),\dq(\blu))
	=\f{( \bmeta(\plus)+\bmeta(\free),
		\bmeta(\minus),
		\bmeta(\free),
		\bmeta(\plus)
		 )}
		{ 2-\bmeta(\minus) }\,,\\
	\hq[(\bmeta_j)_{j\ge1}]
		&=(\hq(\red),
			\hq(\yel)=\hq(\grn)=\hq(\blu))
		= \f{(\bm{\hat{u}}(\plus),
		\bm{\hat{u}}(\free))}
			{3-2\bm{\hat{u}}(\plus)}\,.
	\nonumber
	\end{align}
Let $\mu^\textup{col}$ denote the law of $\dq$,
and let $\hat{\mu}^\textup{col}$
denote the law of $\hq$.
\end{dfn}

\begin{dfn}[Bethe free energy of color model]
\label{d:Phi.col}
Given a sequence of probability measures $\vec{\dq}=(\dq_1,\ldots,\dq_k)$ on $\set{\RYGB}$,
we define a probability measure on $\set{\RYGB}^k$ by
(cf. \eqref{e:clause.tuple.measure.weighted} and \eqref{e:opt.clause.tuple.measure.star})
	\[
	\Big(\hbh[\vec{\dq}]\Big)
		(\sigma_1,\ldots,\sigma_k)
	\equiv 
	\f{\hat{\varphi}_a(\sigma_1,\ldots,\sigma_k)}
	{\hbz[\vec{\dq}]}
	\prod_{i=1}^k
	\dq_i(\sigma_i)\,,
	\]
with $\hat{\varphi}_a$ given by \eqref{e:indicator.of.valid.clause.coloring}.
Similarly, given integers $\ud \equiv (d^\plus,d^\minus)$
with $d=d^\plus + d^\minus$,
and an array $\vec{\hq}
\equiv ((\hq^\plus)_i,(\hq^\minus)_i)_i$
of probability measures
on $\set{\RYGB}$,
we define a probability measure on $\set{\RYGB}^d$ by
(cf.\ \eqref{e:clause.tuple.measure.weighted} and \eqref{e:opt.clause.tuple.measure.star})
	\[\Big(
	\dbh[\ud,\vec{\hq}]\Big)(\sigma_1,\ldots,\sigma_d)
	\equiv \f{\varphi_v(\sigma_1,\ldots,\sigma_d)}
		{\dbz[\ud,\vec{\hq}]}
	\Bigg(
	\prod_{i=1}^{d^\plus}
	(\hq^\plus)_i(\sigma_i)
	\Bigg)
	\Bigg(
	\prod_{i=d^\plus+1}^d
	(\hq^\minus)_i(\sigma_i)
	\Bigg)\,,
	\]
with $\varphi_v$ given by \eqref{e:color.model.variable.factor} for a variable in which the first $d^\plus$ incident edges have the $\plus$ sign, and the remaining $d^\minus$ incident edges have the $\minus$ sign. Finally, given two measures $\dq$ and $\hq$ on $
\set{\RYGB}$, define (cf.\ \eqref{e:edge.marginal.q.times.q}
and \eqref{e:defn.canonical.edge.marginal})
	\[
	\Big(\bar{\nu}[\dq,\hq]\Big)
	(\sigma)
	\equiv \f{\dq(\sigma)\hq(\sigma)}
		{\bar{z}[\dq,\hq]}\,.
	\]
Abbreviate $\POpm$ for the law of $\ud$,
and recall 
$\mu^\textup{col}$
and $\hat{\mu}^\textup{col}$
from Definition~\ref{d:bethe.laws.of.color.messages}.
Define
	\begin{align*}
	H^{\textup{col},\textup{v}}(\alpha)
	&\equiv\int
	\Ent\Big(\dbh[\ud,\vec{\hq}]\Big)
	\,d\POpm(\ud)
	\,d(\hat{\mu}^\textup{col})^\otimes(\vec{\hq})\,,\\
	H^{\textup{col},\textup{cl}}(\alpha)
	&\equiv
	\int\Ent\Big(\hbh[\vec{\dq}]\Big)
	\,d(\mu^\textup{col})^\otimes(\vec{\dq})\,,\\
	H^{\textup{col},\textup{e}}(\alpha)
	&\equiv
	\int\int\Ent\Big(\bar{\nu}[\dq,\hq]\Big)
	\,d\mu^\textup{col}(\dq)
	\,d\hat{\mu}^\textup{col}(\hq)\,.
	\end{align*}
(The dependence on $\alpha$ is through the law of $\ud$ as well as the measures $\mu^\textup{col}$
and $\hat{\mu}^\textup{col}$.) The \bemph{Bethe free energy of the color model} is given by (cf.\ \eqref{e:rate.function.given.gen.degseq})
	\beq\label{e:phi.col.entropic.formula}
	\Phi^\textup{col}(\alpha)
	\equiv
	H^{\textup{col},\textup{v}}(\alpha)
	+ \alpha \bigg\{
	H^{\textup{col},\textup{cl}}(\alpha)
	- k\,
		H^{\textup{col},\textup{e}}(\alpha)
		\bigg\}\,.
	\eeq
An equivalent (and more commonly seen) expression is given by defining
	\begin{align}\nonumber
	\Phi^{\textup{col},\textup{v}}(\alpha)
	&\equiv\int
	\log\dbz[\ud,\vec{\hq}]
	\,d\POpm(\ud)
	\,d(\hat{\mu}^\textup{col})^\otimes(\vec{\hq})\,,\\
	\nonumber
	\Phi^{\textup{col},\textup{cl}}(\alpha)
	&\equiv
	\int \log\hbz[\vec{\dq}]
	\,d(\mu^\textup{col})^\otimes(\vec{\dq})\,,\\
	\Phi^{\textup{col},\textup{e}}(\alpha)
	&\equiv
	\int\int \log\bar{z}[\dq,\hq]
	\,d\mu^\textup{col}(\dq)
	\,d\hat{\mu}^\textup{col}(\hq)\,.
	\label{e:simplified.phi.col.terms}
	\end{align}
It is straightforward to check that
	\beq\label{e:simplified.phi.col}
	\Phi^\textup{col}(\alpha)
	= \Phi^{\textup{col},\textup{v}}(\alpha)
	+ \alpha \bigg\{
	\Phi^{\textup{col},\textup{cl}}(\alpha)
	- k\,\Phi^{\textup{col},\textup{e}}(\alpha)
		\bigg\}\,.
	\eeq
We will show in \S\ref{ss:judicious.first.mmt}
(Lemma~\ref{l:phi.equals.phi.col}) that 
$\Phi^\textup{col}(\alpha)$
is exactly the same as the $\onersb$ free energy
$\Phi(\alpha)$ from \eqref{e:phi.alpha}.
\end{dfn}

In \S\ref{ss:judicious.first.mmt} we will complete the proof of Proposition~\ref{p:first.moment.exponent} by showing that the quantity $\bm{\Psi}_{\DD}(\omega)$ (as defined by \eqref{e:nu.opt}, and appearing in the estimate of Corollary~\ref{c:first.moment.exponent}) is lower bounded by a quantity that tends to $\Phi^\textup{col}(\alpha)=\Phi(\alpha)$ in the limit $n\to\infty$ and $R\to\infty$.

\subsection{Second moment and constrained entropy maximization}
\label{ss:second.moment.overview}

In this subsection we introduce some of the core principles of the proof of Proposition~\ref{p:second.moment.judicious}.
As before, let $\GG'$ denote the original $\ksat$ instance, and $\GG\equiv\proc\GG'$ its processed version. Recall \eqref{e:def.ZZ} that $\ZZ\equiv\ZZ(\GG)$ counts all colorings of $\GG$ that are judicious (Definition~\ref{d:judicious}).
In Proposition~\ref{p:second.moment.judicious} we seek to calculate the expected value, under the measure $\P_{\DD}$, of
	\beq\label{e:second.mmt.decompose.z.IO}
	\ZZ^2[I_0]
	\equiv \sum_{z\in I_0}
	\ZZ^2[z]\,,
	\eeq
for $I_0$ as
defined by \eqref{e:middle.interval},
and $\ZZ^2[z]$ as defined by \eqref{e:second.mmt.decompose.by.z}.

Throughout this subsection, we let
$\usi$ denote a pair $(\usi^1,\usi^2)$
where each $\usi^i$ is a judicious coloring of $\GG$.
Given any such $\usi$, let $\pi$ and $\omega$ be defined analogously to \eqref{e:edge.msr.given.edge.type}
and \eqref{e:edge.msr.given.clause.type}
from Definition~\ref{d:empirical.measures.of.colors},
 except that now edge spins take values in $\set{\RYGB}^2$ rather than $\set{\RYGB}$. We hereafter refer to $\omega\equiv(\omega_{\bL,j})_{\bL,j}$ as the \bemph{pair empirical measure (on edges)}. Let $\ZZ^2(\omega)$ denote the contribution to $\ZZ^2$ from configurations $\usi\equiv(\usi^1,\usi^2)$ with empirical measure $\omega$. Given $\usi$, we can also define the vertex pair empirical measure $\nu\equiv(\dbh,\hbh)$ as in \eqref{e:variable.color.cond.on.type} and \eqref{e:clause.color.cond.on.type}, except with pairs of colors instead of single colors on each edge. Let $\ZZ^2[\nu]$ denote the contribution to $\ZZ^2$ from pairs $\usi$ that are consistent with $\nu$. In the pair coloring model, as in the single-copy model, the edge empirical measure $\omega$ can be determined as a function of the vertex empirical measure $\nu$ (in fact, $\omega$ can be determined from $\hbh$ alone).

\begin{dfn}[judicious pair empirical measures]\label{d:judicious.pair}
As $\ZZ$ is defined to count only judicious configurations, in order for the contribution $\ZZ^2(\omega)$ to be non-zero, $\omega$ must satisfy two properties. First, the single-copy marginals of $\omega$
must agree with the measure $\omstar$
defined in the statement of Corollary~\ref{c:first.moment.exponent}: that is to say, for both $i=1,2$, we must have
	\beq\label{e:margin.judicious}
	\omstar_{\bL,j}(\sigma^i)
	= (\omega_{\bL,j})^j(\sigma^i)
	\equiv
	\sum_{\tau\in\set{\RYGB}^2}
	\Ind{\tau^i=\sigma^i}
	\omega_{\bL,j}(\tau)
	\eeq
for all $\bL,j$. Secondly, $\omega$ must arise as the marginal of a valid vertex measure $\nu=(\dbh,\hbh)$. We say that $\omega$ is \bemph{judicious} if it satisfies both these properties. (The second property is important mainly for computing the dimension of the space of feasible $\omega$, as we will see in the proof of Lemma~\ref{l:dimension.count.two} below.)\end{dfn}

With the above notations, we can refine the above decomposition~\eqref{e:second.mmt.decompose.z.IO} as
	\[
	\ZZ^2[I_0]
	= \sum_{\omega\in\bm{I}_0}
	\ZZ^2(\omega)\,,
	\]
where $\bm{I}_0$ denotes the subset of measures $\omega$ that are judicious and consistent with $z\in I_0$. Throughout what follows, we will use the term ``pair coloring model'' for the two-copy version of \eqref{e:color.factor.model}, with factors
	\begin{align}\nonumber
	\varphi_{v,2}(\usi_{\delta v})
	&\equiv
	\prod_{j=1,2}
	\varphi_v( (\usi^j)_{\delta v}))\,,\\
	\varphi_{a,2}(\usi_{\delta a})
	&\equiv
	\prod_{j=1,2}
	\hat{\varphi}_a
	((\usi^j)_{\delta a})\,.
	\label{e:pair.color.model.factors}
	\end{align}
Let $\bm{\Phi}_{\DD,2}(\nu)$
and $\bm{\Psi}_{\DD,2}(\omega)$
be the analogues of \eqref{e:rate.function.given.gen.degseq}
and \eqref{e:nu.opt} for the pair coloring model.
We then have the following extension of
Lemma~\ref{l:combinatorial.mmt.simplification}: 

\begin{lem}\label{l:if.neg.def}
Denote $\prodom\equiv\omstar\otimes\omstar$ for $\omstar$ as in Corollary~\ref{c:first.moment.exponent}.
Suppose that if we restrict to $\omega\in\bm{I}_0$,
the function
$\bm{\Psi}_{\DD,2}$ is uniquely maximized at $\prodom$, with negative-definite Hessian.
Then 
	\[\E_{\DD}\Big( \ZZ^2[I_0]\Big)
	\asymp_R \Big(\E_{\DD}\ZZ\Big)^2\,,\]
i.e., the conclusion of Proposition~\ref{p:second.moment.judicious} holds.

\begin{proof} The same calculation leading to
\eqref{e:stirling.with.poly.error} gives
(in the interior of the simplex of feasible $\nu$)
a similar formula
	\[
	\E_{\DD}\ZZ^2[\nu]
	\asymp_R
	\f1{ n^{\bm{s}_2}}
	\exp\bigg\{n \bm{\Phi}_{\DD,2}(\nu)\bigg\}\,,
	\]
where $\bm{s}_2=\bm{s}_2(\DD)$ 
takes into account the polynomial corrections from the Stirling approximation,
and is the analogue of \eqref{e:stirling.correction.number} for the pair model:
	\beq\label{e:stirling.correction.two}
	\bm{s}_2=\bm{s}_2(\DD)
	=\sum_{\bT}\bigg( (s_{\bT})^2-1\bigg)
	+\sum_{\bL}\bigg( (s_{\bL})^2-1\bigg)
	-\sum_{\bt}\bigg( (s_{\bt})^2-1\bigg)\,.
	\eeq
Let us write $\nu\sim\omega$ if $\nu$ is consistent with $\omega$. In Lemma~\ref{l:dimension.count.two} (deferred to \S\ref{ss:coherence.weights}), we will calculate
	\beq\label{e:dimension.count.pair.model}
	d_2(\DD)
	\equiv\dim\bigg\{\nu : \nu\sim\omega\bigg\}\,.
	\eeq
for $\omega$ close to $\prodom$. Then, similarly to
\eqref{e:first.mmt.poly}
and \eqref{e:prep.for.gaussian.sum}, we have
	\[
	\E_{\DD}\ZZ^2(\omega)
	=
	\sum_{\nu:\nu\sim\omega}
	\E_{\DD}\ZZ^2[\nu]
	\asymp_R
	\sum_{\nu:\nu\sim\omega}
	\f1{n^{\bm{s}_2/2}}
	\exp\bigg\{n \bm{\Phi}_{\DD,2}(\nu)\bigg\}
	\asymp_R
	\f{n^{d_2/2}}
		{n^{\bm{s}_2/2}}
	\exp\bigg\{
	\bm{\Psi}_{\DD,2}(\omega)
	\bigg\}\,.
	\]
Finally, we will show in Lemma~\ref{l:dimension.of.judicious.omega} (also deferred to \S\ref{ss:coherence.weights}) that
	\beq\label{e:dim.judicious.omega}
	j_2\equiv j_2(\DD)\equiv
	\dim\bigg\{\omega
		:\textup{$\omega$ is judicious}\bigg\}
	=
	\bm{s}_2(\DD)- d_2(\DD)- 2\wp(\DD)
	\eeq
for $\wp$ as in \eqref{e:wp.of.R}. If $\bm{\Psi}_{\DD,2}$ satisfies the conditions of the lemma, then the gaussian summation estimate
\eqref{e:gaussian.sum} gives
	\[\E_{\DD}\ZZ^2[\bm{I}_0]
	\asymp_R
	\f1{n^{\wp}}
	\exp\bigg\{
	\bm{\Psi}_{\DD,2}(\prodom)
	\bigg\}
	\stackrel{\odot}{=}
	\Bigg(
	\f1{n^{\wp/2}}
	\exp\bigg\{
	\bm{\Psi}_{\DD}(\omstar)
	\bigg\}
	\Bigg)^2
	\asymp_R
	\Big(\E_{\DD}\ZZ\Big)^2\,.
	\]
(The step marked $\odot$ uses the identity
$\bm{\Psi}_{\DD,2}(\prodom)
=2\bm{\Psi}_{\DD}(\omstar)$
which is easy to verify.)
\end{proof}
\end{lem}

Thus, in order to prove Proposition~\ref{p:second.moment.judicious}, it suffices to verify the condition of Lemma~\ref{l:if.neg.def}. In the remainder of this subsection, we show that this condition can be reduced to solving a family of constrained entropy maximization problems on finite trees, which we define next. We will separately consider two cases, one for compound enclosures and one for non-compound variables. The solution of the optimization problems occupies most of Sections~\ref{s:contract}--\ref{s:burnin}.\medskip


\noindent\bemph{Entropy maximization problem for compound enclosures.} We discuss the case of compound enclosures first. Although these regions are more complicated in the sense that they contain defective variables, we have an important advantage in that the notion of compound type (Definition~\ref{d:cpd.type}) encodes the structure of the entire enclosure. This allows us to reduce the analysis of each type of compound enclosure to an optimization problem concerning colorings of a fixed tree with fixed edge types (Proposition~\ref{p:update.compound} below). By contrast, to obtain an analogous statement for non-compound variables, we will have to consider a more complicated optimization problem 
where some of the edge types on the tree can vary (Proposition~\ref{p:block.update.non.compound} below).

\begin{dfn}[judicious measures on trees]
\label{d:constrained.opt.compound.enclosure}
Let $U$ be a finite bipartite factor tree,
with all vertices and edges labelled by (mutually compatible) \bemph{compound} total types, such that all the leaves of $U$ are variables.
Let $\pd_\circ U$ be a designated nonempty subset of leaf variables. We use $\delta U$ to denote the edges incident to $\pd_\circ U$. Define the simplex of probability measures
	\[
	\Simplex(U)
	\equiv \left\{
	\hspace{-3pt}\begin{array}{c}
	\textup{probability measures $\nu$
		over pair}\\
	\textup{colorings 
		$\usi\equiv(\usi^1,\usi^2)$
		of $U$}
	\end{array}\hspace{-3pt}\right\}\,.
	\]
We say that $\nu\in\Simplex(U)$ is \bemph{judicious}
if all of its edge marginals match the canonical marginals of Definition~\ref{d:canonical}: that is, for all edges $e$ in $U$ and for both $j=1,2$, 
we have
	\[
	\starpi_e(\sigma^j)
	=(\nu_e)^j(\sigma^j)
	\equiv \nu((\sigma_e)^j=\sigma^j)
	\equiv \sum_\tau
		\Ind{\tau^j=\sigma^j}
		\nu(\tau_e=\tau)
	\]
for all $\sigma^j\in\set{\RYGB}$. Note that, since all the edges in $U$ are assumed to be of compound type, this is equivalent to saying that all the edge marginals $\nu_e$ satisfy the condition \eqref{e:margin.judicious} from Definition~\ref{d:judicious.pair}. Let
$\omega$ denote a tuple $(\omega_e)_e$ 
where 
$\omega_e$ is a judicious probability measure on $\set{\RYGB}^2$ for each edge $e$ in $U$. We then let
	\[
	\Judicious(U ; \omega_{\delta U})
	\equiv
	\left\{
	\hspace{-3pt}\begin{array}{c}
	\nu\in\Simplex(U):
	\textup{$\nu$ is judicious,}\\
	\textup{and
	$\nu_e=\omega_e$ for all $e\in\delta U$}
	\end{array}\hspace{-3pt}
	\right\}\,,
	\]
where $\nu_e$ denotes the marginal of $\nu$ on edge $e$. We also let
	\[\Simplex(U;\omega_U)
	\equiv
	\bigg\{\nu\in\Simplex(U) : \nu_e=\omega_e
		\textup{ for all }e\in E_U\bigg\}\,,
	\]
and note that
$\Simplex(U;\omega_U)
\subseteq \Judicious(U ; \omega_{\delta U})
\subseteq \Simplex(U)$
for any judicious $\omega$.
\end{dfn}

\begin{ppn}[block optimization for compound regions]
\label{p:update.compound}
In the processed graph $\GG=(V,F,E)=\proc\GG'$, let $U\equiv (V_U,F_U,E_U)$ be any compound enclosure. Recall from Definition~\ref{d:enclosure}
that $U= U^\circ \cup \pd U^\circ$ where 
$\pd U^\circ\equiv\pd_\circ U$ is the set of perfect variables in $U$. Let $\delta U$ denote the edges in $U$ that are incident to $\pd_\circ U$, and decompose
	\[
	\omega
	\equiv\begin{pmatrix}
	\omega_\textup{in}\\
	\omega_{\delta U}\\
	\omega_\textup{out}
	\end{pmatrix}
	\equiv
	\begin{pmatrix}
	(\omega_{\bL,j}
	: \textup{some edge in $E_U\setminus \delta U$
		has type $(\bL,j)$})
	\\
	(\omega_{\bL,j}
	: \textup{some edge in $\delta U$
		has type $(\bL,j)$})\\
	(\omega_{\bL,j}
	: \textup{some edge in $E \setminus E_U$
		has type $(\bL,j)$})
	\end{pmatrix}\,.
	\]
(An edge $e=(av)$ has type $(\bL,j)$ if $\bL_a=\bL$ and $e$ is the $j$-th edge in $\delta a$.)
We also abbreviate 
$\omega_U\equiv(\omega_\textup{in},\omega_{\delta U})$ and
$\omega_\textup{bd}\equiv(\omega_{\delta U},\omega_\textup{out})$. Assume the neighborhood profile $\DD$ of $\GG$ is bounded away from zero in the sense of \eqref{e:pos.frac}. Then, using the notation of Definition~\ref{d:constrained.opt.compound.enclosure}, we have
	\begin{align}
	\label{e:max.over.acute.omega.compound}
	&\max_{\acute{\omega}}
	\Bigg\{
	\bm{\Psi}_{\DD,2}(\acute{\omega})
	-\bm{\Psi}_{\DD,2}(\omega)
	: 
	\textup{$\omega$ and $\acute{\omega}$
	are judicious,
	$\acute{\omega}_\textup{bd}=\omega_\textup{bd}$}
	\Bigg\}\\
	&=c_\textup{in}
	\Bigg( \max_\nu
	\Bigg\{ \Ent(\nu)
		: 	\nu \in
		\Judicious(U,\omega_{\delta U})\Bigg\}
	- \max_\nu
	\Bigg\{ \Ent(\nu)
		: \nu \in \Simplex(U;\omega_U)\Bigg\}
	\Bigg)\,,
	\end{align}
where $c_\textup{in}$ depends only on $\DD$ and on $U$, and is lower bounded by $c_1$.

\begin{proof} Recalling Remark~\ref{d:degseq},
let $\cD$ be any neighborhood sequence that is consistent with $\DD$. Since for any $\omega$ we have 
$\E_{\DD}\ZZ^2(\omega)
=\E_\cD\ZZ^2(\omega)$,
rather than working under $\P_{\DD}$
we can instead work under $\P_\cD$,
which can be sampled by the (generalized) configuration model described in
Remark~\ref{r:processed.graph.types.CM}.
Let $\bt(U)$ denote the collection of all edge types $\bt$ appearing on edges inside $U$; recall that these are all compound types.

Let $\mathfrak{M}$ denote the random matching of $\delta V$ to $\delta F$ (which defines the graph $\GG$). We let $\mathfrak{m}$
denote the restriction of $\mathfrak{M}$ to 
half-edges with types in $\bt(U)$; this defines a subgraph $\UU \equiv(V_{\UU},F_{\UU},E_{\UU}) \subseteq\GG$ which consists of $n_U = n c_\textup{in}$ disjoint copies of $U$. Moreover, since $\UU$ 
consists of $n_U$ disjoint copies of $U$ for any
valid realization of
$\mathfrak{m}$, we can condition on 
$\mathfrak{m}$ without changing the expected value of the partition function:
	\[\f1{n^{O(1)}}
	\exp\Big\{
	n\bm{\Psi}_{\DD,2}(\omega)
	\Big\}= \E_{\DD}\ZZ^2(\omega)
	=\E_\cD\ZZ^2(\omega)
	=\E_\cD\bigg(
		\ZZ^2(\omega)
		\,\bigg|\, \mathfrak{m}
		\bigg)\,.
	\]
Let $\pd_\circ\UU$ denote the $n_U$ copies of $\pd_\circ U$ inside $\UU$, and let $\delta\UU$ denote the edges in $\UU$ incident to $\pd_\circ\UU$. Any (pair) coloring $\usi$ on $\GG$ can be decomposed as
	\[
	\usi
	\equiv 
	\begin{pmatrix}
	\usi_\textup{in} \\
	\usi_{\delta\UU} \\
	\usi_\textup{out}
	\end{pmatrix}
	\equiv
	\begin{pmatrix}
	(\sigma_e : e \in E_{\UU}\setminus \delta\UU) \\
	(\sigma_e : e \in \delta\UU) \\
	(\sigma_e : e \in E\setminus E_{\UU}) 
	\end{pmatrix}\,.
	\]
Note that this corresponds precisely to the above decomposition of $\omega$: a coloring $\usi$ has empirical measure $\omega$ (which we abbreviate $
\usi\sim\omega$) if and only if
$\usi_\textup{in}$,
$\usi_{\delta\UU}$, and
$\usi_\textup{out}$
have empirical measures
$\omega_\textup{in}$,
$\omega_{\delta U}$, and
$\omega_\textup{out}$ respectively.
We also abbreviate
$\usi_{\UU}
	\equiv ( \usi_\textup{in},
		\usi_{\delta\UU})$
	and 
	$\omega_U
	\equiv (\omega_\textup{in},
	\omega_{\delta U})$.
With this notation, we can decompose
$\ZZ^2(\omega)$ as
	\beq\label{e:partition.fn.tree.factorize}
	\ZZ^2(\omega)
	=
	\sum_{\usi_{\UU}}
		\Ind{\usi_{\UU} \sim \omega_U}
	\varphi_{\UU}(\usi_{\UU})
	\Bigg\{
	\sum_{\usi_\textup{out}}
		\Ind{\usi_\textup{out} \sim\omega_\textup{out}}
	\varphi_\textup{out}
	(\usi_{\delta\UU},\usi_\textup{out})
	\Bigg\}\,,
	\eeq
where, recalling the definition \eqref{e:pair.color.model.factors}
of the factors for the pair coloring model, we define
	\begin{align*}
	\varphi_{\UU}(\usi_{\UU})
	&\equiv
	\prod_{v\in V_{\UU}\setminus\pd_\circ U}
	\varphi_v(\usi_{\delta v})
	\prod_{a\in F_{\UU}}
	\hat{\varphi}_a(\usi_{\delta a})\,, \\
	\varphi_\textup{out}
	(\usi_{\delta\UU},\usi_\textup{out})
	&\equiv
	\prod_{v\in (V\setminus V_U) \cup \pd_\circ U}
	\varphi_v(\usi_{\delta v})
	\prod_{a\in F\setminus F_{\UU}}
	\hat{\varphi}_a(\usi_{\delta a})\,.
	\end{align*}
Taking the expectation conditional on $\mathfrak{m}$ gives
	\[
	\E_\cD\Big(\ZZ^2(\omega)\,\Big|\,\mathfrak{m}\Big)
	=
	\sum_{\usi_{\UU}}
		\Ind{\usi_{\UU} \sim \omega_U}
	\varphi_{\UU}(\usi_{\UU})\\
	\overbrace{
	\E_\cD
	\Bigg(
	\sum_{\usi_\textup{out}}
		\Ind{\usi_\textup{out} \sim\omega_\textup{out}}
	\varphi_\textup{out}
	(\usi_{\delta\UU},\usi_\textup{out})
	\,\Bigg|\,\mathfrak{m}
	\Bigg)}^{X\equiv X(\omega_{\delta U},\omega_\textup{out})}\,,\]
where we emphasize that the value of 
$X$ is constant over all $\usi_{\delta\UU}\sim\omega_{\delta U}$. Now recall that $\UU$ consists of a disjoint union of copies of $U$, which we denote $U_i$ for $1\le i\le n_U$. For any $\usi_{\UU}$ we can define its empirical measure over the copies of $U$, that is to say,
	\[
	\nu(\usi_U)
	= \f1{n_U}
	\sum_{i=1}^{n_U}
	\Ind{ \usi_{ U_i} = \usi_U}
	\]
We write $\usi_{\UU}\sim\nu$ if $\usi_{\UU}$ has empirical measure $\nu$. Then, with $\Simplex(U;\omega_U)$ as in Definition~\ref{d:constrained.opt.compound.enclosure}, we have
	\begin{align*}
	\sum_{\usi_{\UU}}
		\Ind{\usi_{\UU} \sim \omega_U}
	\varphi_{\UU}(\usi_{\UU})
	&=
	\sum_{\nu\in\Simplex(U;\omega_U)}
	\sum_{\usi_{\UU}}
		\Ind{\usi_{\UU} \sim \nu}
	\varphi_{\UU}(\usi_{\UU})
	=
	\sum_{\nu\in\Simplex(U;\omega_U)}
	\binom{n_U}{n_U\nu}\\
	&= n^{O(1)} \exp\Bigg\{ n_U
	\max_\nu \bigg\{
	\Ent(\nu) : \nu \in \Simplex(U;\omega_U)
		\bigg\}\Bigg\}\,,
	\end{align*}
where the last step is by Stirling's approximation. Altogether we conclude
	\beq\label{e:reduce.to.simplex.omega}
	\exp\bigg\{n \bm{\Psi}_{\DD,2}(\omega)\bigg\}
	= n^{O(1)}
	X\exp\Bigg\{ n_U
	\max_\nu \bigg\{
	\Ent(\nu) : \nu \in \Simplex(U;\omega_U)
		\bigg\}\Bigg\}
	\eeq
On the other hand, summing over all possibilities of $\omega_\textup{in}$ gives
	\begin{align}\nonumber
	&\f1{n^{O(1)}}
	\exp\Bigg\{
		n\max\Big\{
	\bm{\Psi}_{\DD,2}(\acute{\omega})
	: \acute{\omega}_\textup{bd}=\omega_\textup{bd}
	\Big\}
	\Bigg\}
	=
	\sum_{\acute{\omega}}
	\Ind{\acute{\omega}_\textup{bd}=\omega_\textup{bd}}
	\E_\cD\bigg(
		\ZZ^2(\acute{\omega})
		\,\bigg|\, \mathfrak{m}
		\bigg)\\
	&= X
	\sum_{\usi_{\UU}}
	\Ind{\usi_{\delta\UU} \sim \omega_{\delta U}}
	\varphi_{\UU}(\usi_{\UU})
	= n^{O(1)}
	X
	\exp\Bigg\{ n_U
	\max_\nu \bigg\{
	\Ent(\nu) : \nu \in \Judicious(U;\omega_{\delta U})
		\bigg\}\Bigg\}\,.
	\label{e:reduce.to.simplex.judicious}
	\end{align}
The claim follows by combining \eqref{e:reduce.to.simplex.omega} with \eqref{e:reduce.to.simplex.judicious}.
\end{proof}
\end{ppn}

\noindent\bemph{Entropy maximization problem for non-compound variables.} We now give the analogues of Definition~\ref{d:constrained.opt.compound.enclosure} and Proposition~\ref{p:update.compound} for the case of non-compound variables. As noted above, we now have the added difficulty that for a variable of non-compound type $\bT$, the clause types $\bL$ neighboring to the variable are not uniquely determined by $\bT$. This difficulty will be countervailed by the fact that non-compound variables are perfect, and as a result it will be sufficient to consider only the depth-one neighborhood of the variable.

\begin{dfn}[judicious measures on trees with augmented alphabet]
\label{d:judicious.augmented.alphabet}
Let $v$ be a variable of total type $\bT$, which we assume is not of compound type. Let $U\equiv U_{\bT}$ be the depth-one neighborhood of $v$, in which each edge $e\in\delta v$ is labelled with its corresponding type $\bt_e$. However, we \bemph{forget the total type labellings on the clauses and other edges of $U$,} and each boundary edge $e\in\delta U$ is labelled only with its index $j(e)\in[k]$. An \bemph{augmented (pair) coloring} on $U$ is a configuration $(\usi,\uL)$ which assigns to each edge $e\in U$ a spin $(\sigma_e,\bL_e)$ where $\sigma_e\in\set{\RYGB}^2$ and $\bL_e$ is a clause total type. Recall \eqref{e:pair.color.model.factors} where we defined the variable and clause factors $\varphi_v$ and $\hat{\varphi}_a$ for the pair coloring model. The factors for the augmented pair coloring model are 
	\begin{align}\nonumber
	\varphi_v(\usi_{\delta v},\uL_{\delta v})
	&\equiv
	\varphi_v(\usi_{\delta v})
	\prod_{e\in\delta v}
	\Ind{\bL_e\ni\bt_e}\,,\\
	\hat{\varphi}_a(\usi_{\delta a},\uL_{\delta a})
	&\equiv
	\hat{\varphi}_a(\usi_{\delta a})
	\mathbf{1}\bigg\{
	\textup{the $\bL_e$ are the same for all
		$e\in\delta a$}
	\bigg\}\,.
	\label{e:augmented.model.factors}
	\end{align}
We say that $(\usi,\uL)$ is a valid
augmented coloring on $U$ as long as
	\[
	\varphi_v(\usi_{\delta v},\uL_{\delta v})
	\prod_{a\in\pd v}
	\hat{\varphi}_a(\usi_{\delta a},\uL_{\delta a})
	=1\,.
	\]
Analogously to the 
simplex $\Simplex(U)$ from
Definition~\ref{d:constrained.opt.compound.enclosure}, 
we now let
	\[
	\Simplex_\textup{aug}(U)
	\equiv\left\{
	\hspace{-3pt}\begin{array}{c}
	\textup{probability measures $\nu$
		on valid}\\
	\textup{augmented pair colorings
		$(\usi,\uL)$ of $U$}
	\end{array}\hspace{-3pt}
	\right\}\,.
	\]
For each edge $e$ in $U$, let $p(e)$ denote the edge in $\delta v$ that is closest to $e$. (If $e\in\delta v$ then $p(e)=e$.)
We then say that a measure $\nu\in\Simplex_\textup{aug}(U)$ is \bemph{fully judicious with respect to $\DD$} if it holds for all edges $e$ in $U$ that
	\beq\label{e:pi.DD.first.appearance}
	\nu_e(\sigma,\bL)
	= \overbrace{
	\underbrace{
	\f{\DS
	\hat{\DD}(\bL)\Ind{\bL_{j(\bt_{p(e)})}
		=\bt_{p(e)}}}
		{	\DS \sum_{\bL'}
	\hat{\DD}(\bL')\Ind{(\bL')_{j(\bt_{p(e)})}
		=\bt_{p(e)}}}
	}_{\textup{denote this }
	\pi_{\DD}(\bL\,|\,\bt_{p(e)})}
	\tilde{\omega}_{\bL,j(e)}(\sigma)
	}^{\textup{denote this 
	$(\AUGMENT_{\DD}(\tilde{\omega}))_e(\sigma,\bL)$}}
	\eeq
for some $\tilde{\omega}$ that is judicious
(in the sense of Definition~\ref{d:judicious.pair}). Given any particular $\omega$, we now write
	\[
	\omega_{\delta U}
	\equiv
	\Bigg(
	\begin{array}{c}
	\omega_{\bL,j}
	: \textup{some edge $e\in\delta U$
		has $j(e)=j$,
		and can take clause}\\
	\textup{type $\bL_e=\bL$
	in the augmented coloring model}
	\end{array}
	\Bigg)\,.
	\]
With this notation, we can define (compare with $\Judicious(J;\omega_{\delta U}($ from Definition~\ref{d:constrained.opt.compound.enclosure}) 
	\[
	\Judicious_{\DD}(U;\omega_{\delta U})
	\equiv
	\bigg\{
	\hspace{-3pt}\begin{array}{c}
	\nu\in\Simplex_\textup{aug}(U)
	: \textup{$\nu$ is fully judicious
	with respect to $\DD$,}\\
	\textup{and $\nu_e= (\AUGMENT_{\DD}(\omega))_e$
	for all $e\in\delta U$}
	\end{array}\hspace{-3pt}
	\bigg\}\,,\]
where $\nu_e$ now denotes the marginal law under $\nu$ of $(\sigma_e,\bL_e)$. We also let
	\[
	\Simplex_{\DD}(U;\omega_U)
	\equiv\bigg\{
	\nu\in\Simplex_\textup{aug}(U)
	: \nu_e = (\AUGMENT_{\DD}(\omega))_e
	 \textup{ for all }e\in U
	\bigg\}\,,
	\]
and note that
$\Simplex_{\DD}(U;\omega_U)\subseteq
\Judicious_{\DD}(U;\omega_{\delta U})
\subseteq\Simplex_\textup{aug}(U)$
for any judicious $\omega$.
\end{dfn}

\begin{ppn}[block optimization for non-compound variables]
\label{p:block.update.non.compound}
Let $\bT$ be any variable total type that appears in the processed graph $\GG=(V,F,E)=\proc\GG'$. Assume that $\bT$ is not of compound type. Define
the subgraph
	\[\UU\equiv
	\bigcup_{v:\bT_v=\bT}
	B_1(v;\GG)
	\equiv(V_{\UU},F_{\UU},E_{\UU})
	\subseteq\GG\,,\]
and let $\delta\UU$ denote the leaf edges of $\UU$. Decompose
	\[
	\omega
	\equiv\begin{pmatrix}
	\omega_\textup{in}\\
	\omega_{\delta U}\\
	\omega_\textup{out}
	\end{pmatrix}
	\equiv
	\begin{pmatrix}
	(\omega_{\bL,j}
	: \textup{some edge in 
		$E_{\UU}\setminus \delta\UU$
		has type $(\bL,j)$})
	\\
	(\omega_{\bL,j}
	: \textup{some edge in $\delta\UU$
		has type $(\bL,j)$})\\
	(\omega_{\bL,j}
	: \textup{some edge in $E \setminus E_{\UU}$
		has type $(\bL,j)$})
	\end{pmatrix}\,.
	\]
We also abbreviate
$\omega_U
\equiv(\omega_\textup{in},\omega_{\delta U})$ and
$\omega_\textup{bd}\equiv(\omega_{\delta U},\omega_\textup{out})$. Assume the neighborhood profile $\DD$ of $\GG$
is bounded away from zero in the sense of \eqref{e:pos.frac}. Then, using the notation of Definition~\ref{d:judicious.augmented.alphabet}, we have
\begin{align*}
	&\max_{\acute{\omega}}
	\Bigg\{
	\bm{\Psi}_{\DD,2}(\acute{\omega})
	-\bm{\Psi}_{\DD,2}(\omega)
	: 
	\textup{$\omega$
	 and $\acute{\omega}$ are judicious,
	$\acute{\omega}_\textup{bd}=\omega_\textup{bd}$}
	\Bigg\}\\
	&=c_\textup{in}
	\Bigg( \max_\nu
	\Bigg\{ \Ent(\nu)
		: \nu \in
		\Judicious_{\DD}(U,\omega_{\delta U})\Bigg\}
	- \max_\nu
	\Bigg\{ \Ent(\nu)
		: \nu \in \Simplex_{\DD}(U;\omega_U)\Bigg\}
	\Bigg)\,,
	\end{align*}
where $c_\textup{in}$ depends only on $\DD$ and on $\bT$, and is lower bounded by $c_1$.

\begin{proof} As in the proof of Proposition~\ref{p:update.compound},
we can fix a neighborhood sequence $\cD$ that is consistent with $\DD$, and work under the measure $\P_\cD$. Let $\bt(\bT)$
denote the set of all edge types that can be incident to a variable of type $\bT$. Let $\bL(\bT)$ denote the set of all clause types that can neighbor a variable of type $\bT$:
	\[
	\bL(\bT)
	= \bigg\{ \bL
	: \bt\in\bL \textup{ for some }
		\bt\in \bt(\bT) \bigg\}\,.
	\]
Let $\mathfrak{M}$ denote the random matching of $\delta V$ to $\delta F$ (which defines the graph $\GG$). We let $\mathfrak{m}_1$ denote the restriction of $\mathfrak{M}$ to the half-edges of types belonging to $\bt(\bT)$. We then let $\mathfrak{m}_2$ denote the edges in $\mathfrak{M}\setminus\mathfrak{m}_1$ that are incident to the edges of $\mathfrak{m}_1$. Finally, we let $\mathfrak{m}_3=\mathfrak{M}\setminus(\mathfrak{m}_1\cup\mathfrak{m}_2)$.
The partial matching $(\mathfrak{m}_1,\mathfrak{m}_2)$ defines the subgraph $\UU\subseteq\GG$ consisting of the depth-one neighborhoods of all the variables of type $\bT$. Analogously to \eqref{e:partition.fn.tree.factorize}, we have the decomposition
	\[\ZZ^2(\omega)
	=
	\sum_{\usi_{\UU}}
		\Ind{\usi_{\UU} \sim \omega_U}
	\varphi_{\UU}(\usi_{\UU})
	\Bigg\{
	\sum_{\usi_\textup{out}}
		\Ind{\usi_\textup{out} \sim\omega_\textup{out}}
	\varphi_\textup{out}
	(\usi_{\delta\UU},\usi_\textup{out})
	\Bigg\}\,.\]
Note that $\varphi_{\UU}$ depends on $(\mathfrak{m}_1,\mathfrak{m}_2)$, while
$\varphi_\textup{out}$ depends on 
$(\mathfrak{m}_2,\mathfrak{m}_3)$. 
However, we have
	\[\E_\cD\Bigg(
	\sum_{\usi_\textup{out}}
		\Ind{\usi_\textup{out} \sim\omega_\textup{out}}
	\varphi_\textup{out}
	(\usi_{\delta\UU},\usi_\textup{out})
	\,\Bigg|\,\mathfrak{m}_2\Bigg)
	= X(\omega_{\delta U},\omega_\textup{out})
	\equiv X
	\]
for any $\usi_{\delta\UU}\sim\omega_{\delta U}$. 
Under $\P_\cD$, the matching $\mathfrak{m}_1$
is independent of the pair $(\mathfrak{m}_2,\mathfrak{m}_3)$. The expected value of $\ZZ^2(\omega)$ under $\P_\cD$ remains the same even after conditioning on $\mathfrak{m}_2$, so we have
	\begin{align*}
	\E_\cD\ZZ^2(\omega)
	&=\E_\cD\Big( \ZZ^2(\omega)\,\Big|\,
		\mathfrak{m}_2\Big)
	= X \cdot
	\E_\cD\Bigg( \sum_{\usi_{\UU}}
		\Ind{\usi_{\UU} \sim \omega_U}
	\varphi_{\UU}(\usi_{\UU})
	\,\Bigg|\,\mathfrak{m}_2
	\Bigg)\\
	&=
	X \cdot \sum_{\mathfrak{m}_1}
	\P_\cD(\mathfrak{m}_1)
	\underbrace{
	\sum_{\usi_{\UU}}
		\Ind{\usi_{\UU} \sim \omega_U}
	\varphi_{\UU}(\usi_{\UU})}_{\textup{depends on $(
		\mathfrak{m}_1,\mathfrak{m}_2)$}}\,.
	\end{align*}
Recall that $(\mathfrak{m}_1,\mathfrak{m}_2)$ defines the subgraph $\UU$ consisting of the depth-one neighborhoods of all the variables of type $\bT$; without loss we suppose those variables are labelled 
$\set{v_1,\ldots,v_{n_{\bT}}}$. On the other hand, let $\mathcal{U
}$ denote the graph consisting of $n_{\bT}$ disjoint copies $U_1,\ldots,U_{n_{\bT}}$ of the tree $U$ from Definition~\ref{d:judicious.augmented.alphabet}. A tuple $(\mathfrak{m}_1,\mathfrak{m}_2,\usi_{\UU})$
with $\usi_{\UU}\sim\omega_U$ can be mapped bijectively to an augmented coloring
$(\usi_\mathcal{U},\bL_\mathcal{U})$
of $\mathcal{U}$ with edge empirical measures
$\AUGMENT(\omega_U)$. (The bijection goes as follows: given $(\mathfrak{m}_1,\mathfrak{m}_2)$ there is a unique graph isomorphism $g :\UU\to \mathcal{U}$ which maps $v_i$ to the root of $U_i$ for each $i$, respects edge types $\bt_e$ for all $e\in\delta v_i$ for all $i$, and respects the edge indices $j(e)$ for all $e\in\delta\UU$. The coloring $\usi_{\UU}$ is mapped under $g$ to a coloring on $\UU$. Finally, for each edge $e=(au)$ in $U$, we set $\lit_e$ to be the clause type of $g^{-1}(a)$.) It follows that
	\[
	\sum_{\mathfrak{m}_1}
	\P_\cD(\mathfrak{m}_1)
	\sum_{\usi_{\UU}}
		\Ind{\usi_{\UU} \sim \omega_U}
	\varphi_{\UU}(\usi_{\UU})
	=
	\f1{|\set{\mathfrak{m}_1}|}
	\sum_{(\usi_\mathcal{U},\bL_\mathcal{U})}
	\Ind{
	(\usi_\mathcal{U},\bL_\mathcal{U}) \sim
	\AUGMENT(\omega_U)}
	\varphi_{\mathcal{U}}(
	\usi_\mathcal{U},\bL_\mathcal{U})\,,\]
where $\set{\mathfrak{m}_1}$ denotes the set of all matchings $\mathfrak{m}_1$ that are consistent with $\cD$. It follows analogously to 
\eqref{e:reduce.to.simplex.omega} that
	\[
	\exp\bigg\{
	n \bm{\Psi}_{\DD,2}(\omega)
	\bigg\}
	=\f{n^{O(1)} X }{|\mathfrak{m}_1|}
	\exp\Bigg\{n_{\bT}
	\max_\nu\bigg\{\Ent(\nu):
		\nu\in\Simplex_{\DD}(U;\omega_U)
	\Bigg\}\,.\]
On the other hand, summing over all possibilities of $\omega_\textup{in}$ gives, analogously to \eqref{e:reduce.to.simplex.judicious},
	\[
	\exp\Bigg\{
	n \max\bigg\{
	\bm{\Psi}_{\DD,2}(\acute{\omega})
	: \acute{\omega}_\textup{bd}=\omega_\textup{bd}
	\bigg\}\Bigg\}
	=\f{n^{O(1)} X }{|\mathfrak{m}_1|}
	\exp\Bigg\{n_{\bT}
	\max_\nu\bigg\{\Ent(\nu):
		\nu\in \Judicious_{\DD}(U;\omega_{\delta U})
	\Bigg\}\,.
	\]
The claim follows.
\end{proof}
\end{ppn}
\medskip

\noindent\bemph{Contraction estimates and coordinate descent.} We now give some informal discussion of how Propositions~\ref{p:update.compound} and \ref{p:block.update.non.compound} are used in the proof of the key second moment estimate Proposition~\ref{p:second.moment.judicious}. The details of the proof are rather complicated, and are laid out in Sections~\ref{s:contract}--\ref{s:burnin}. However, some of the high-level ideas are rather simple, and we point them out here. As above, let $U$ denote a compound enclosure, or the depth-one neighborhood of a perfect variable. For the purposes of this discussion, we will express the pair empirical measure $\omega=(\omega_{\bL,j})_{\bL,j}$ as a tuple $y=(y_1,\ldots,y_\ell)$ where each $y_i$ denotes a subset of entries of $\omega$ corresponding to edges in the interior $U^\circ$ of $U$. Thus $\ell$ is the number of distinct choices of $U$, where ``distinct'' here means that the types are distinct. Write $f(y)\equiv\bm{\Psi}_{\DD,2}(\omega)$, and then let $\prody\equiv(\prody_1,\ldots,\prody_\ell)$ denote the $y$ that corresponds to $\prodom$. As explained in 
Lemma~\ref{l:if.neg.def}, the conclusion of Proposition~\ref{p:second.moment.judicious} follows if we can show that in a neighborhood of $\prody$ (corresponding to $\omega\in\bm{I}_0$) the function $f$ is uniquely maximized at $\prody$, with negative-definite Hessian.

With the above notation, we see that Propositions~\ref{p:update.compound} and \ref{p:block.update.non.compound} explain how to optimize $f(y)$ in a single coordinate $y_i$, keeping the other coordinates $y_{-i} \equiv (y_1,\ldots,y_{i-1},y_{i+1},\ldots,y_\ell)$ fixed. Moreover, if $y_{-i}$ is fixed, then the optimization over $y_i$ is a entropy maximization problem (constrained to affine subspaces corresponding to $\Judicious(U;\omega_{\delta U})$ or $\Judicious_{\DD}(U;\omega_{\delta U})$), which means the function is strictly concave in $y_i$ if $y_{-i}$ is fixed. Given $y$, let
 	\beq\label{e:tilde.y.block.update}
	\tilde{y}_i
	\equiv \tilde{y}_i(y_{-i})
	= \argmax_{x_i}
	\bigg\{ f(x_i, y_{-i})\bigg\}\,,
	\eeq
i.e., $\tilde{y}_i$ is the result of optimizing the $i$-th coordinate keeping the others fixed. A key estimate that we will prove is that if $y$ is close enough to $\prody$, then this update brings $\tilde{y}_i$ closer to $\prody_i$:
	\beq\label{e:contraction.simplified}
	\|\tilde{y}_i-\prody_i\|
	\le \f{\|y_i-\prody_i\|}{2}\,.
	\eeq
We have not specified the norm $\|\cdot\|$ for which the above holds --- in fact we obtain contraction estimates for a ``discrepancy'' measure \eqref{e:def.discrepancy.measure.e} which is not quite a norm, but is close enough to serve our purpose. Even ignoring this issue, the bound \eqref{e:contraction.simplified} is a simplification of the precise contraction results that we obtain, which are characterized by 
Propositions~\ref{p:contraction.for.simple.var} and \ref{p:contraction.COMPOUND}.

Assuming the simplified estimate
\eqref{e:contraction.simplified}, it is straightforward to argue that in a small neighborhood of $y=\prody$, the function $f(y)$ is uniquely maximized at $\prody$ --- for any $y\ne\prody$, we can apply the update \eqref{e:tilde.y.block.update} in some coordinate $y_i\ne\prody_i$, and the value of $f$ will increase because $f$ is strictly convex in each individual coordinate $y_i$. Having shown this, we can proceed as follows: for $\tilde{y}_i$ as defined by \eqref{e:tilde.y.block.update}, we have
	\[
	\bigg( f(\prody)-f(y)\bigg)^{1/2}
	\ge \bigg(f(\tilde{y}_i,y_{-i})-f(y)\bigg)^{1/2}
	\gtrsim 
	\|\tilde{y}_i-y_i\|
	\ge \|y_i-\prody_i\|-\|\tilde{y}_i-\prody_i\|
	\stackrel{\eqref{e:contraction.simplified}}{\ge}
	\f{ \|y_i-\prody_i\|}{2}\,,
	\]
which shows that $f$ has negative-definite Hessian at $\prody$. (We remark again that this is a simplified sketch of the actual proof, which is more complicated because we do not have such a simple estimate as \eqref{e:contraction.simplified}. The detailed proof of the negative-definite Hessian condition appears in \S\ref{ss:hess}.) 

Based on the above discussion, we have the following proof strategy. First, show that if we restrict to
$\omega\in\bm{I}_0$, then the maximum of $\bm{\Psi}_{\DD,2}(\omega)$ can only be attained in a small neighborhood of $\omega=\prodom$. We call this step the ``a priori estimate,'' and it is deferred to Section~\ref{s:burnin}. Then show that in the small neighborhood of $\omega=\prodom$
(i.e., $y=\prody$), we have (some version of) 
the contraction estimate \eqref{e:contraction.simplified}. The contraction estimate occupies the majority of Sections~\ref{s:contract} and \ref{s:merge}, and is the most essential technical ingredient of the proof. Combining these steps with the above argument will show that 
if we restrict to $\omega\in\bm{I}_0$, then
$\bm{\Psi}_{\DD,2}(\omega)=f(y)$ is uniquely maximized at $\omega=\prodom$ with negative-definite Hessian. Thus the conditions of Lemma~\ref{l:if.neg.def} are satisfied, and Proposition~\ref{p:second.moment.judicious} follows.

We conclude this subsection with some discussion of the key contraction estimate \eqref{e:contraction.simplified}. Recall that $\tilde{y}_i$ is defined by \eqref{e:tilde.y.block.update}, or equivalently, by the constrained entropy maximization problems
from Propositions~\ref{p:update.compound} and \ref{p:block.update.non.compound}. To be concrete, consider the compound case from Proposition~\ref{p:update.compound}. Roughly speaking, the approach will be to find a weight $\Lambda(\usi)$ (where $\usi$ is a pair coloring of the enclosure $U$)
which is ``Lagrangian'' in the sense that $\langle\log\Lambda,\nu\rangle$ is constant over $\nu\in\Judicious(U;\omega_{\delta U})$, and which turns the constrained maximization problem into an unconstrained one:
	\[
	\nu
	=\argmax\Bigg\{
	\Ent(\nu) : \nu\in\Judicious(U;\omega_{\delta U})
	\Bigg\}
	=\argmax
	\Bigg\{
	\Ent(\nu) + \langle\log\Lambda,\nu\rangle
	: \nu\in\Simplex(U)
	\Bigg\}\,.
	\]
By calculus, the solution is given simply by $\nu(\usi)\cong\Lambda(\usi)$, the $\Lambda$-weighted Gibbs measure on $U$. Thus, if $q$ denotes the \textsc{bp} messages for the $\Lambda$-weighted model, we can easily read off edge marginals of $\nu$ from the usual formula $\nu_e\cong \dq_e\hq_e$. The basic strategy of the proof of \eqref{e:contraction.simplified} is to show that
\bemph{in the pair coloring model, the the \textsc{bp} recursion contracts towards the product message $\prodq\equiv \qbul\otimes\qbul$, provided we start close enough to $\prodq$.} As a consequence, if the boundary condition $\omega_{\delta U}$ is close to product, we will be able to construct weights $\Lambda$ that are close to $\Lmstar\otimes\Lmstar$ (for $\Lmstar$ from Corollary~\ref{c:clause.bp.weights}). Then, for $e$ in the interior of $U$, the messages $\dq_e$ and $\hq_e$ will be closer to product than the messages on $\delta U$, so the discrepandy between $\omega_e\cong\proddq_e\prodhq_e$ and $\prodom_e$ will be closer than the discrepancies on $\delta U$. This gives the rough idea of the proof of \eqref{e:contraction.simplified}, and we leave the details to Sections~\ref{s:contract}--\ref{s:merge}.

\subsection{Weights for coherent clauses}
\label{ss:coherence.weights}

In this subsection we prove a series of claims concerning weakly and strictly coherent clauses (Definition~\ref{d:coherence}). The key implications are that $\SQpi$ is always strictly coherent, so $\starpi$ will also be strictly coherent if it is ``close enough'' to $\SQpi$. When $\starpi$ is strictly coherent, we will show that clauses can be reweighted such that \textsc{bp} equations hold. First, however, we make a remark on the possibilities of $\supp\starpi_e$:

\begin{rmk} The purpose of this remark is to emphasize that the canonical marginal $\starpi_e$ can in general be supported a strict subset of $\set{\RYGB}$. In fact, although $\supp\starpi_e$ always contains the \SPIN{green} spin, $(\supp\starpi_e)\setminus\grn$ can be any subset of $\set{\red,\yel,\blu}$. 
For example, suppose in a $\ksat$ instance $\GG=(V,F,E)$ that a clause $a\in F$ has among its neighbors two leaf variables $u\ne v$. Then both $u$ and $v$ must always be \SPIN{free}, so
	\[
	\supp\starpi_{au}=\supp\starpi_{av}=\set{\grn}\,.
	\]
The clause $a$ can never be forcing, so for any other variable $w\in\pd a\setminus\set{u,v}$ we must have
$\red\notin\supp\starpi_{aw}$. However, such $w$ can be forced by other clauses $b\in\pd w\setminus a$ --- indeed, $\supp\starpi_{aw}$ contains $\set{\blu}$ if and only if $w$ may be forced by some $b\in\pd w(\plus\lit_{aw})\setminus a$; and it contains $\set{\yel}$ if and only if $w$ may be forced by some $b\in\pd w(\minus\lit_{aw})$. Therefore, in this scenario, $\supp\starpi_{aw}\setminus\grn$ can be any subset of $\set{\yel,\blu}$. Similarly, on edges $e\in E$ where forcing can occur (meaning $\red\in\supp\starpi_e$), it is easy to construct examples to see that $\supp\starpi_e\setminus\set{\grn,\red}$ can be 
 any subset of $\set{\yel,\blu}$.
\end{rmk}

We now turn to the main results of the subsection.

\begin{lem}\label{l:coher.feasibility}
For a clause $a\in F$, a tuple $\pi=(\pi_e)_{e\in\delta a}$
is weakly coherent if and only if there exists a probability measure $\nu_{\delta a}$ over valid colorings of $\delta a$ with edge marginals $\pi$: that is,
	\[\pi_e(\sigma)
	= \sum_{\usi_{\delta a}}
	\Ind{\sigma_e=\sigma}
	\nu_{\delta a}(\usi_{\delta a})\]
for all $e\in\delta a$ and all $\sigma\in\set{\red,\yel,\grn,\blu}$.

\begin{proof}
Recall that a valid coloring of $\delta a$ is a tuple $\usi\equiv \usi_{\delta a}\in\set{\red,\yel,\grn,\blu}^{\delta a}$ for which $\hat{\varphi}_a(\usi)=1$, as defined by \eqref{e:indicator.of.valid.clause.coloring} with the grouping $\cya=\set{\grn,\blu}$. If $|\delta a|=k$, then a valid coloring $\usi\in\set{\red,\yel,\grn,\blu}^k$ falls in one of two cases: (i) it has exactly one $\red$ entry with the remaining entries $\yel$; or (ii) all entries are in $\set{\cya,\yel}$ and at least two entries are $\cya$. If $\nu$ is a probability measure on valid colorings $\usi\in\set{\red,\yel,\grn,\blu}^k$, let
$(R_i,Y_i,G_i,B_i)$ be the associated marginal probabilities of $\red,\yel,\grn,\blu$ on each coordinate $i \in \set{1,\ldots,k}$, so for instance
	\[
	R_i \equiv \nu(\sigma_i=\red)
		= \sum_{\usi\in\set{\red,\yel,\grn,\blu}^k}
		\Ind{\sigma_i=\red}
		\nu(\usi)\,.
	\]
Thus, for each $i\in[k]\equiv\set{1,\ldots,k}$, the quantities $R_i,Y_i,G_i,B_i$ are nonnegative and sum to one. This lemma is purely a claim about the feasible polytope of edge marginals coming from probability measures over valid colorings: the assertion is that for any $k\ge3$, the following are equivalent:
\begin{enumerate}[(a)]
\item \label{i:polytope.linconds} The tuple $(R_i,Y_i,G_i,B_i)_{1\le i\le k}$ satisfies the following constraints: for each $i\in[k]$, the quantities $R_i,Y_i,G_i,B_i$ are nonnegative and sum to one. Moreover, with $R\equiv R_1+\ldots+R_k$, we have $Y_i\ge R-R_i$ for each $i$. Lastly, writing $C_i \equiv G_i + B_i$, and $C\equiv C_1+\ldots+C_k$, we have $C \ge 2(1-R)$.
\item \label{i:polytope.probcond} The marginals $(R_i,Y_i,G_i,B_i)_{1\le i\le k}$ can be realized by a probability measure on valid colorings $\usi\in\set{\red,\yel,\grn,\blu}^k$.
\end{enumerate} 
Let us first verify the straightforward direction, that
\eqref{i:polytope.probcond} implies \eqref{i:polytope.linconds}. If $\nu$ is a probability measure on valid colorings $\usi\in\set{\red,\yel,\grn,\blu}^k$, then for each coordinate $i$ it is clear that the marginal probabilities $R_i,Y_i,G_i,B_i$ are nonnegative and sum to one. Next, since the color $\red$ can only occur together with $k-1$ entries $\yel$, we have
$R_i = \nu(\yel^{i-1}\red\yel^{k-i})$ for each $i\in[k]$. This implies 
	\[
	Y_i
	\ge \sum_{j\in[k]\setminus i}
		\nu(\yel^{j-1}\red\yel^{k-j})
	= \sum_{j\in[k]\setminus i} R_j = R-R_i\,.
	\]
On the other hand, since any valid $\usi$ with no $\red$ entry must be in $\set{\yel,\cya}^k$ with at least two $\cya$ entries, we have
	\beq\label{e:total.mass.cyan.weak.ineq.proof}
	\sum_{i=1}^k C_i
	=
	\sum_{\usi\in \set{\yel,\cya}^k}
	\nu(\usi)
	\sum_{i=1}^k \Ind{\sigma_i=\cya}
	\ge 2
	\sum_{\usi\in \set{\yel,\cya}^k}
	\nu(\usi)
	= 2(1-R)\,.
	\eeq
This proves that \eqref{i:polytope.probcond} implies \eqref{i:polytope.linconds}.

In the converse direction, given $(R_i,Y_i,G_i,B_i)_{1\le i\le k}$ satisfying the conditions of \eqref{i:polytope.linconds}, we now describe one particular construction of a measure $\nu$ on valid colorings that realizes these marginals.\footnote{The total number of valid colorings $\usi\in\set{\RYGB}^k$ is $3^k-k-1$, so $\nu$ ranges over a $(3^k-k-2)$-dimensional simplex. On the other hand, since $R_i+Y_i+G_i+B_i=1$ for all $i$, the tuple $(R_i,Y_i,G_i,B_i)_{1\le i\le k}$ is restricted to an affine space of dimension $3k$. Thus, simply by comparing dimensions, we would expect that any generic feasible tuple $(R_i,Y_i,G_i,B_i)_{1\le i\le k}$ can be realized by an uncountable (and convex) family of measures $\nu$.} First, it is clear that we must set $\nu(\yel^{i-1}\red\yel^{k-i})=R_i$ for each $i\in[k]$. This step assigns $R$ of the probability mass of $\nu$, where the conditions in \eqref{i:polytope.linconds} ensure that $0\le R \le R_i + Y_i \le 1$. If $R=1$ then we are done, so assume otherwise, meaning there is a positive amount $1-R$ of mass left over that remains to be assigned. On each $i$ let $y_i$ denote the marginal weight of $\yel$ left over after the first step:
	\[
	y_i = Y_i - \sum_{j\in[k]\setminus i} R_j
	= Y_i + R_i - R\,,
	\]
which is nonnegative by \eqref{i:polytope.linconds}. The total mass left over on each edge $i$ is $C_i+y_i=1-R$. Let $\bm{C}_1,\ldots,\bm{C}_k$ be consecutive intervals of length $C_i$,
	\[\bm{C}_i=\bigg[ \sum_{j=1}^{i-1} C_j,
		\sum_{j=1}^i C_j \bigg)\,.\]
These intervals give a partition of $[0,C)$, which in turn is a subset of $[0,k(1-R))$. Let $J\equiv[0,1-R)$, and let $p$ be the mapping from $\mathbb{R}$ to $J$ which sends each real number to its representative modulo $(1-R)\mathbb{Z}$ in $J$. Let $J_{i,\cya}$ be the image of $\bm{C}_i$ under $p$ (i.e., the consecutive intervals get ``wrapped around''), and note that the restriction of $p$ to $\bm{C}_i$ is one-to-one since $|\bm{C}_i|=C_i\le 1-R$. Let $J_{i,\grn}$ be any subset of $J_{i,\cya}$ with Lebesgue measure $G_i$, and let $J_{i,\blu} \equiv J_{i,\cya} \setminus J_{i,\grn}$. Let $J_{i,\yel}\equiv J\setminus J_{i,\cya}$. For $t\in J$ let $\sigma_i(t)$ be the unique element $\sigma\in\set{\yel,\grn,\blu}$ such that $t\in J_{i,\sigma}$. Let $\usi(t)\equiv(\sigma_i(t))_{i\in[k]}$. For $\usi\in\set{\yel,\grn,\blu}^k$, let
	\[
	\nu(\usi)
	\equiv
	\textup{Leb}\bigg(
	\Big\{t\in J : \usi(t) = \usi \Big\}
	\bigg)
	= \textup{Leb}
	\bigg( \bigcap_{i=1}^k J_{i,\sigma_i}
	\bigg)\,,
	\]
where $\textup{Leb}$ denotes Lebesgue measure. This completes the definition of $\nu$. It is immediate from the construction that $\nu$ has marginals $(R_i,Y_i,G_i,B_i)_{1\le i\le k}$, and all configurations in its support are of form $\yel^{i-1}\red\yel^{k-i}$, or lie in $\set{\yel,\grn,\blu}^k$. It remains to check that all colorings of the latter case are valid, which is to say that $\usi(t)$ has at least two $\cya$ entries for every $t\in J$. By definition, $\sigma_i(t)=\cya$ if and only if $t\in J_{i,\cya} = p(\bm{C}_i)$, which occurs if and only if $t + \ell(1-R) \in \bm{C}_i$ for some integer $\ell$. It follows that the number of $\cya$ entries in $\usi(t)$ is
	\[
	\sum_{i=1}^k\sum_{\ell\in\mathbb{Z}}
	\mathbf{1}\bigg\{t + \ell(1-R) \in \bm{C}_i \bigg\}
	= \sum_{\ell\in\mathbb{Z}}
	\mathbf{1}\bigg\{
	t + \ell(1-R) \in [0,C) \bigg\}\ge2\,,
	\]
where the last inequality holds for all $t\in J\equiv [0,1-R)$ using the final condition $C\ge2(1-R)$ from \eqref{i:polytope.linconds}. This shows that $\nu$ is supported on valid colorings, thereby concluding our proof that \eqref{i:polytope.linconds} implies \eqref{i:polytope.probcond}.\end{proof}\end{lem}

\begin{lem}\label{l:clause.based.marginals.cohere} For an acyclic clause $a\in F$, the clause-based marginals $\SQpi=(\SQpi_e)_{e\in\delta a}$ are strictly coherent for all $r\ge2$.

\begin{proof}
Recall from Remark~\ref{r:first.rmk.coherence}
that the probability measure
	\[
	\nu_{\delta a}(\usi_{\delta a})
	=\f1{\bm{\hat{z}}_a}
	\hat{\varphi}_a(\usi_{\delta a})
	\prod_{e\in\delta a}
	\SQdq(\sigma_e)
	\]
has marginals $\SQpi$. It therefore follows immediately from Lemma~\ref{l:coher.feasibility} that $\SQpi$ is weakly coherent. Recall from \eqref{e:defn.canonical.edge.marginal} that $\SQpi_e(\sigma)$ is proportional $\SQdq_e(\sigma)\SQhq_e(\sigma)$. It follows from the correspondence \eqref{e:color.recursions.eta} that 
$\SQdq_e(\grn)$ and $\SQhq_e(\grn)$ are positive for all $r\ge2$,
so $\SQpi_e(\grn)$ is positive on any edge $e$.
Now suppose $\SQpi_e(\yel)$ is positive: this implies that 
$\SQdq_e(\yel)$ and $\SQhq_e(\yel)$ must both be positive, therefore
	\[
	\SQpi_e(\yel)
	-\sum_{e'\in\delta a\setminus e}
	\SQpi_{e'}(\red)
	\ge
	\nu_{\delta a}\bigg(
	\textup{$\sigma_e=\yel$, and 
	$\sigma_{e'}=\grn$ for all 
	$e'\in\delta a\setminus e$}
	\bigg)
	= \f1{\bm{\hat{z}}_a}
	\SQdq_e(\yel)
	\prod_{e'\in\delta a\setminus e}
	\SQdq_{e'}(\grn)>0\,,
	\]
which shows that $\SQpi$ satisfies \eqref{e:cohere.y.for.r} with strict inequality whenever $\SQpi_e(\yel)$.
Next, recalling the proof of Lemma~\ref{l:coher.feasibility},
we see that
\eqref{e:total.mass.cyan.weak.ineq.proof}
holds with equality if and only if $\nu_{\delta a}$
gives zero mass to configurations $\usi_{\delta a}$ with more than two $\cya$ entries. This does not happen in the current situation, since
	\[
	\nu_{\delta a}(\sigma_e=\grn
		\textup{ for all }e\in\delta a)
	= \f1{\bm{\hat{z}}_a}
	\prod_{e\in\delta a}
	\SQdq_e(\grn)>0\,.
	\]
This proves that $\SQpi$ always satisfies
\eqref{e:cohere.enough.cyan} with strict inequality. The claim follows.
\end{proof}
\end{lem}

\begin{lem}\label{l:clause.strict.coherence}
For a clause $a\in F$, if the tuple $\pi\equiv (\pi_e)_{e\in\delta a}$ is strictly coherent,
then there exists a measure $\bar{\nu}_{\delta a}$ with edge marginals $\pi_e$ for all $e\in\delta a$, and 
$\supp\bar{\nu}_{\delta a}=\CCOLS_{\delta a}$ where
	\[\CCOLS_{\delta a}
	\equiv\bigg\{
	\textup{valid colorings }\usi_{\delta a}
	\,\bigg|\, \sigma_e\in\supp\pi_e
	\textup{ for all }e\in\delta a
	\bigg\}\,.
	\]

\begin{proof} It follows from Lemma~\ref{l:coher.feasibility} that there is a probability measure $\nu_{\delta a}$ with edge marginals $\pi\equiv(\pi_e)_{e\in\delta a}$, which of course implies $\supp \nu_{\delta a}\subseteq\CCOLS_{\delta a}$. If $\CCOLS_{\delta a}$ is a singleton then we must have $\supp \nu_{\delta a}\subseteq\CCOLS_{\delta a}$, in which case the assertion follows immediately by taking $\bar{\nu}_{\delta a}=\nu_{\delta a}$. We thus assume for the remainder of the proof that $\CCOLS_{\delta a}$ has at least two elements.

Now let $\MARG_{\delta a}$ denote the space of \emph{all} edge marginals $\pi\equiv (\pi_e)_{e\in\delta a}$ which can arise from probability measures over $\CCOLS_{\delta a}$. It follows from Lemma~\ref{l:coher.feasibility} that $\MARG_{\delta a}$ is characterized by the conditions \eqref{e:cohere.y.for.r} and \eqref{e:cohere.enough.cyan}, together with the constraints imposed by the supports of the $\pi_e$. To describe this more explicitly, let
	\[
	R(\pi)
	\equiv\sum_{e\in\delta a}\pi_e(\red)\,,\quad
	C(\pi)
	\equiv\sum_{e\in\delta a}\pi_e(\cya)\,.
	\]
We then divide the scenarios into three cases, according to the number of edges $e\in\delta a$ with $\pi_e(\red)>0$:
\begin{enumerate}[(i)]
\item If $|\set{e\in\delta a : \pi_e(\red)>0}|\ge2$, then we must have $\pi_e(\yel)>0$ for all $e\in\delta a$, and
	\[
	\MARG_{\delta a}
	\equiv\left\{
	\pi
	\,\left|\,
	\begin{array}{l}
	\pi_e(\yel) \ge R(\pi) - \pi_e(\red)
	\textup{ for all } e\in\delta a\,, \\
	C(\pi) \ge 2[1-R(\pi)]\,,\\
	\supp\pi_e\subseteq\supp\pi_e
	\textup{ for all } e\in\delta a
	\end{array}\right.
	\right\}\,.
	\]

\item If there is a unique edge $e'\in\delta a$ with $\pi_{e'}(\red)>0$, then
	\[
	\MARG_{\delta a}
	\equiv\left\{
	\pi
	\,\left|\,
	\begin{array}{l} 
	\pi_e(\yel) \ge R(\pi)= \pi_{e'}(\red)
	\textup{ for all }
	e\in\delta a \setminus \set{e'} \,, \\
	C(\pi) \ge 2[1-R(\pi)]
		=2[1-\pi_{e'}(\red)]
		\,,\\
	\supp\pi_e\subseteq\supp\pi_e
	\textup{ for all } e\in\delta a
	\end{array}\right.
	\right\}\,.
	\]

\item If $|\set{e\in\delta a : \pi_e(\red)>0}|=0$, then
	\[
	\MARG_{\delta a}
	\equiv\left\{
	\pi
	\,\left|\,
	\begin{array}{l}
	C(\pi) \ge 2[1-R(\pi)] = 2\,,\\
	\supp\pi_e\subseteq\supp\pi_e
	\textup{ for all } e\in\delta a
	\end{array}\right.
	\right\}\,.
	\]
\end{enumerate}
In all cases, the definition of strict coherence ensures that $\pi\in\RelInt\MARG_{\delta a}$, the relative interior of $\MARG_{\delta a}$. Now let $u_{\delta a}$ be the uniform measure on $\CCOLS_{\delta a}$, and let $\unifpi$ be the marginals resulting from $u_{\delta a}$. Since $\unifpi\in\MARG_{\delta a}$ and $\pi\in\RelInt\MARG_{\delta a}$, it must hold for sufficiently small $\ep$ that
	\[
	\pi^\ep\equiv
	\f{\pi-\ep\unifpi}{1-\ep}
	\in \MARG_{\delta a}\,.
	\]
By Lemma~\ref{l:coher.feasibility} there is a probability measure $\nu^\ep$ on $\CCOLS_{\delta a}$ with edge marginals $\pi^\ep$. Then $\bar{\nu}_{\delta a}
\equiv (1-\ep)\nu^\ep + \ep u_{\delta a}$
is fully supported on $\CCOLS_{\delta a}$, with edge marginals
	\[
	(1-\ep)\pi^\ep + \ep \pi^\textup{unif}
	=(1-\ep)
	\f{\pi-\ep\unifpi}{1-\ep}
	 + \ep \pi^\textup{unif}
	= \pi\,.
	\]
This concludes the proof.
\end{proof}
\end{lem}

\begin{cor}\label{c:clause.strict.coherence.lagrange}
For a clause $a\in F$, if the tuple $\pi\equiv(\pi_e)_{e\in\delta a}$ is strictly coherent and further satisfies $\pi_e(\grn)>0$ for all $e\in\delta a$, then there exists a set of edge weights $w_e : \set{\RYGB} \to [0,\infty)$ such that the $w$-weighted measure on valid colorings of $\delta a$ has edge marginals consistent with $\pi$. Moreover, if we fix $w_e(\grn)\equiv1$ for all $e\in\delta a$ and require that $\supp w_e \subseteq \supp\pi_e$, then $w$ is unique.

\begin{proof} We will obtain the weights under the assumption that $w_e(\grn)\equiv1$, and $w_e(\sigma)=0$ if $\sigma\notin\supp\starpi_e$; this will imply the result. We can thus restrict our attention to $\PROB_{\delta a}$, the space of all probability measures over $\CCOLS_{\delta a}$. In the case that $\CCOLS_{\delta a}$ consists of a single element, that element must be the all-$\grn$ coloring. Then $\PROB_{\delta a}$ also consists of a single element, which is the probability measure fully supported on the all-$\grn$ coloring, and setting $w_e(\sigma)\equiv\Ind{\sigma=\grn}$ gives the unique weights satisfying the stated conditions.

We therefore assume for the remainder of the proof that $\CCOLS_{\delta a}$ has at least two elements. 
Recall that $\MARG_{\delta a}$ denotes the space of all edge marginals $\pi\equiv (\pi_e)_{e\in\delta a}$ which can arise from probability measures over $\CCOLS_{\delta a}$. We now claim that the space $\MARG_{\delta a}$ has dimension
	\[
	\dim\MARG_{\delta a}
	= D_a =
	\sum_{e\in\delta a}
	\bigg\{|\supp\starpi_e|-1\bigg\}
	\,.\]
Indeed, it is clear that $\dim\MARG_{\delta a} \le D_a$. Equality follows from the strict coherence assumption: let $\acute{\pi}$ be any tuple of measures $(\acute{\pi}_e)_{e\in\delta a}$, subject \emph{only} to the constraints that $\acute{\pi}_e$ and $\pi_e$ have the same support for all $e\in\delta a$, and moreover that $\|\pi-\acute{\pi}\|_\infty\le\delta$. The set of such perturbations $\acute{\pi}$ has dimension $D_a$. Since $\pi$ is strictly coherent, we have that $\acute{\pi}$ is also coherent for small enough $\delta$. It then follows from Lemma~\ref{l:coher.feasibility} that there is a probability measure $\acute{\nu}$ on $\CCOLS_{\delta a}$ with marginal $\acute{\pi}$, which means that $\acute{\pi}\in\MARG_{\delta a}$. Since $\acute{\pi}$ goes over a $D_a$-dimensional space, this proves $\dim\MARG_{\delta a}=D_a$.

Let $A$ be the linear mapping that takes $\nu\in\PROB_{\delta a}$ to its edge marginals $\pi_e(\sigma)$, for $e\in\delta a$ and $\sigma\in\supp\pi_e\setminus\set{\grn}$. We can regard $A$ as a matrix with row indices $\set{(e,\sigma') : e\in\delta a, \sigma'\in\supp\pi_e\setminus\set{\grn}}$, and column indices $\CCOLS_{\delta a}$. The entry of $A$ with row index $(e,\sigma')$ and column index $\usi_{\delta a}$ is the indicator that $\sigma_e = \sigma'$. The image of $\PROB_{\delta a}$ under $A$ is in one-to-one correspondence with $\MARG_{\delta a}$, whose dimension $D_a$ exactly equals the number of rows of $A$. This shows that $A$ is of full rank. Moreover, let $\MSR_{\delta a}$ be the space of all nonnegative \emph{measures} (not necessarily normalized to unit mass) over $\CCOLS_{\delta a}$. The image of $\MSR_{\delta a}$ under the matrix
	\beq\label{e:matrix.A.prime.clause}
	A'\equiv
	\begin{pmatrix}
	A \\ \mathbf{1}^\st
	\end{pmatrix}\eeq
is the space of all nonnegative scalar multiples of elements of $\MARG_{\delta a}$, and has dimension $D_a+1$. It follows that the matrix $A'$ is also of full rank.

Now consider the constrained entropy maximization problem
	\beq\label{e:primal}
	\OPT\equiv\max\bigg\{
	-\sum_{\usi_{\delta a}\in \CCOLS_{\delta a}}
	\nu(\usi_{\delta a})
	\log
	\nu(\usi_{\delta a})
	\,\bigg|\, 
	\nu \in [0,\infty)^{\CCOLS_{\delta a}}
	\textup{ with }
	A\nu = \pi
	\textup{ and }
	\langle \mathbf{1},\nu\rangle = 1
	\bigg\}\,,\eeq
as well as its Lagrange dual: for $(\mathbf{s},\bar{s})$
where $\mathbf{s}\equiv(s_{e,\sigma'}:e\in\delta a, \sigma'\in\supp\pi_e\setminus\set{\grn})\in\mathbb{R}^{D_a}$ and $\bar{s}\in\mathbb{R}$, 
	\beq\label{e:lagrange.dual}
	g(\mathbf{s},\bar{s})
	= \max\bigg\{
	-\sum_{\usi_{\delta a}\in \CCOLS_{\delta a}}
	\nu(\usi_{\delta a})
	\log
	\nu(\usi_{\delta a})+
	\langle \mathbf{s},
		A\nu-\pi
		\rangle
	+ \bar{s}( \langle\mathbf{1},\nu\rangle -1)
	\,\bigg|\, 
	\nu \in [0,\infty)^{\CCOLS_{\delta a}}\bigg\}\,.\eeq
By the strict coherence condition together with Lemma~\ref{l:clause.strict.coherence}, there exists $\bar{\nu}_{\delta a} \in (0,\infty)^{\CCOLS_{\delta a}}$ satisfying $A\bar{\nu}_{\delta a} = \pi$. This implies that the constraints of the (primal) optimization problem
\eqref{e:primal} are feasible, so $\OPT$ is well-defined. Since the domain $(0,\infty)^{\CCOLS_{\delta a}}$ is an open set, it follows by a classical result (see e.g.\ \cite[Thm.~28.2]{MR1451876}) that this problem enjoys strong Lagrange duality. This means that there exists $(\mathbf{S},\bar{S})$ (\emph{a~priori}, not necessarily unique) such that
	\beq\label{e:lagrange.dual.problem}
	\OPT= g(\mathbf{S},\bar{S})
	= \min_{\mathbf{s},\bar{s}}
	\bigg\{
	g(\mathbf{s},\bar{s})
	: \mathbf{s} \in\mathbb{R}^{D_a},
		\bar{s} \in \mathbb{R}
	\bigg\}\,,\eeq
that is to say, the Lagrange dual problem (the right-hand side of \eqref{e:lagrange.dual.problem}) achieves the same value as the constrained (primal) problem \eqref{e:primal}. Now, for any fixed $(\mathbf{s},\bar{s})$, the value \eqref{e:lagrange.dual} of $g(\mathbf{s},\bar{s})$ is given by optimizing a strictly concave objective over all $\nu\in[0,\infty)^{\CCOLS_{\delta a}}$, so it is attained by a unique maximizer $\nu[\mathbf{s},\bar{s}]$. Likewise, in the primal problem \eqref{e:primal}, since the entropy is strictly concave, the value of $\OPT$ must be uniquely attained by some measure $\optnu$. In the definition~\ref{e:lagrange.dual} of $g(\mathbf{s},\bar{s})$, by substituting $\optnu$ into the objective, we see that $g(\mathbf{s},\bar{s})\ge \OPT$. If equality holds, then 
$\nu[\mathbf{s},\bar{s}]$ must equal 
$\optnu$. This shows that
$\nu[\mathbf{S},\bar{S}]$ solves \eqref{e:primal}.

We next argue that the pair $(\mathbf{S},\bar{S})$ is unique. To this end, note that for any $(\mathbf{s},\bar{s})$, the maximizer $\nu[\mathbf{s},\bar{s}]$ cannot occur on the boundary of $\nu\in[0,\infty)^{\CCOLS_{\delta a}}$, because
	\[
	\lim_{\nu(\usi_{\delta a})\downarrow0}
	\f{-\nu(\usi_{\delta a})
	\log\nu (\usi_{\delta a})}
		{\nu(\usi_{\delta a})} = \infty\,.
	\]
Therefore $\nu[\mathbf{s},\bar{s}]$ is the unique stationary point of the objective in $(0,\infty)^{\CCOLS_{\delta a}}$, and we can solve for it by differentiation with respect to $\nu$. It gives, for all $\usi_{\delta a}\in\CCOLS_{\delta a}$,
	\[
	\Big(\nu[\mathbf{s},\bar{s}]\Big)(\usi_{\delta a})
	= \exp\bigg\{
	(A^\st\mathbf{s})_{\usi_{\delta a}}
	+\bar{s}-1\bigg\}
	= \f{\exp(\bar{s})}{e}
	\prod_{e\in\delta a} \exp(s_{e,\sigma_e})\,,
	\]
where for the final expression to make sense we define $s_{e,\grn}=0$. The corresponding objective value is
	\[
	g(\mathbf{s},\bar{s})
	=\sum_{\usi_{\delta a}\in\CCOLS_{\delta a}}
		\exp\bigg\{
	(A^\st\mathbf{s})_{\usi_{\delta a}}
		+\bar{s}-1\bigg\}
	 -\langle \mathbf{s},\pi\rangle-\bar{s}
	=\sum_{\usi_{\delta a}\in\CCOLS_{\delta a}}
	\exp\bigg\{
	(A')^\st \begin{pmatrix}
		\mathbf{s} \\ \bar{s}
		\end{pmatrix}-1
	\bigg\}
	-\langle \mathbf{s},\pi\rangle-\bar{s} \,,
	\]
which is clearly strictly convex as a function of 
$(A')^\st(\mathbf{s},\bar{s})$. Since $(A')^\st$
is a $|\set{\CCOLS_{\delta a}}|\times (D_a+1)$
matrix of rank $D_a+1$, it defines an injective mapping,
so we conclude that $g$ is also strictly convex as a function of the Lagrange variables $(\mathbf{s},\bar{s})$.
 This shows that the Lagrange dual problem
\eqref{e:lagrange.dual.problem} has a unique minimizer $(\mathbf{S},\bar{S})$.

In summary, we have shown that
$\nu[\mathbf{S},\bar{S}]=\optnu$, the unique solution of the primal problem \eqref{e:primal}. Moreover, for any $(\mathbf{s},\bar{s})\ne (\mathbf{S},\bar{S})$ the measure
$\nu[\mathbf{s},\bar{s}]$ does not satisfy the constraints of \eqref{e:primal}: supposing that it did, the fact that $(\mathbf{S},\bar{S})$ is the unique minimizer of $g$ would give
	\[
	g(\mathbf{S},\bar{S})
	<g(\mathbf{s},\bar{s})
	= -\sum_{\usi_{\delta a}\in \CCOLS_{\delta a}}
	\Bigg\{
	\Big(\nu[\mathbf{s},\bar{s}]\Big)(\usi_{\delta a})
	\Bigg\}
	\log\Bigg\{
	\Big(\nu[\mathbf{s},\bar{s}]\Big)(\usi_{\delta a})
	\Bigg\}
	\,,
	\]
which gives a contradiction since $\nu[\mathbf{S},\bar{S}]=\optnu$. The claimed result follows by setting $w_e(\sigma)=\exp(s_{e,\sigma})$ for all $e\in\delta a$, $\sigma\in\supp\pi_e$.
\end{proof}
\end{cor}

As a byproduct of the considerations in the above proof, we can also compute the dimensions claimed in the proofs of 
Corollary~\ref{c:first.moment.exponent} and Lemma~\ref{l:if.neg.def}:

\begin{lem}[{proof of \eqref{e:dim.nu.given.omega}}]\label{l:dimension.count}
In the setting of Corollary~\ref{c:first.moment.exponent} we have $d_1(\DD)=\bm{s}(\DD)-\wp(\DD)$. 

\begin{proof}
Recall from
\eqref{e:dim.nu.given.omega}
that $d_1(\DD)$ counts the dimension of the space $\set{\nu:\nu\sim\omstar}$.
Recall from Definition~\ref{d:pi.omega.nu.notation} that $\nu\equiv(\dbh,\hbh)$ is the empirical measure of vertex colorings.
Since $\omstar$ is fixed, we can treat each vertex type separately:
if $d_1(\bT)$ is the dimension of the space of variable empirical measures $\dbh_{\bT}$ that are consistent with $\omstar$,
and $d_1(\bL)$ is the dimension of the space of clause empirical measures $\hbh_{\bL}$ that are consistent with $\omstar$, then
	\beq\label{e:dimension.one.decompose.into.types}
	d_1(\DD)=\sum_{\bL}d_1(\bL)+\sum_{\bT} d_1(\bT)\,.\eeq
For any given clause type $\bL$, let $A'$ be the matrix from \eqref{e:matrix.A.prime.clause}, which we showed to be of full rank
in the proof of Corollary~\ref{c:clause.strict.coherence.lagrange}.
The space of $\hbh_{\bL}$ that are consistent with $\omstar$ is an affine shift of the kernel of $A'$ (intersected with the simplex of probability measures). It follows that
	\[
	d_1(\bL)
	=\dim\ker A'
	=s_{\bL} - \rk A'
	=s_{\bL} -1 - \sum_j (s_{\bL(j)}-1)\,.
	\]
To make the analogous calculation for a variable type $\bT$, we argue in a few steps:
\begin{enumerate}[a.]
\item First we need the analogue of the (strict) coherence condition for a collection of empirical measures around a variable $v$
 (rather than around a clause, as in Definition~\ref{d:coherence}). We shall say that $(\pi_e)_{e\in\delta v}$ is \bemph{weakly coherent} if there exists a probability measure $\mu$ on $\set{\minus,\plus,\free}$ such that
	\[\bigg(\pi_e(\pur),\pi_e(\yel),\pi_e(\grn)\bigg)
	= \bigg(\mu(\plus\lit_e),\mu(\minus\lit_e),\mu(\free)\bigg)\]
for all $e\in\delta v$, and for both $\lit\in\set{\minus,\plus}$ we have
	\beq\label{e:variable.coherence.red}
	\sum_{e\in\delta v(\lit)} \pi_e(\red) \ge \mu(\lit)\,.
	\eeq
We say that $(\pi_e)_{e\in\delta v}$ is \bemph{strictly coherent} if $\mu(\free)$ is strictly positive and \eqref{e:variable.coherence.red}
holds with strict inequality for both $\lit\in\set{\minus,\plus}$. By a very similar argument as in Lemma~\ref{l:clause.based.marginals.cohere}, the canonical measures $(\starpi_e)_{e\in\delta v}$ are strictly coherent. 

\item By a similar argument as in Lemma~\ref{l:clause.strict.coherence}, if $(\pi_e)_{e\in\delta v}$ is weakly coherent then there exists a measure $\dbh$ on variable colorings $\usi_{\delta v}$ with edge marginals $\pi_e$ for all $e\in\delta v$. Indeed, clearly, $\dbh$ should give weight $\mu(\free)$ to the all-\SPIN{green} coloring. For $\lit\in\set{\minus,\plus}$, the measure $\dbh$ must give weight $\mu(\lit)$ to the colorings that are \SPIN{purple} on all of $\delta v(\plus\lit)$ and \SPIN{yellow} on all of $\delta v(\minus\lit)$. The only non-trivial step is to separate \SPIN{purple} into \SPIN{red} and \SPIN{blue}, and this can be done by repeating the ``wrapping around consecutive intervals'' construction from Lemma~\ref{l:coher.feasibility}.

\item Similarly as in Corollary~\ref{c:clause.strict.coherence.lagrange}, we define a matrix (the variable analogue of \eqref{e:matrix.A.prime.clause})
	\beq\label{e:matrix.B.prime.variable}
	B' = \begin{pmatrix} B \\ \mathbf{1}^\st\end{pmatrix}
	\eeq
which encodes the edge marginal constraints. The columns of the matrix are indexed by the valid colorings $\usi_{\delta v}$ around the variable. 
For each frozen spin $\lit\in\set{\minus,\plus}$ that the variable can take, the matrix $B$ also has a row which is the indicator on the colorings that are \SPIN{purple} on $\delta v(\plus\lit)$. 
For each edge $e\in\delta v$ where both colors $\set{\red,\blu}$ can appear, the matrix $B$ has a row
with entries $\Ind{\sigma_e=\red}$. Since $(\starpi_e)_{e\in\delta v}$ is strictly coherent, we can repeat the argument of 
Corollary~\ref{c:clause.strict.coherence.lagrange} to show that $B'$ is a full rank matrix.
\end{enumerate}
For each edge type $\bt\in\bT$, let $r_{\bt}=\Ind{\set{\red,\blu}\subseteq\supp\starpi_{\bt}}$. It follows that
	\[
	d_1(\bT)
	= \dim\ker B'
	=s_{\bT}-\rk B'
	=s_{\bT}-1-x_{\bT}-\sum_{\bt\in\bT} r_{\bt}\,,
	\]
where we recall from the discussion around \eqref{e:wp.of.R} that $x_{\bT}\in\set{0,1,2}$ counts the number of frozen spins that can appear on a variable of type $\bT$. Substituting into \eqref{e:dimension.one.decompose.into.types} gives an expression for $d_1(\DD)$. Combining with our earlier calculation 
\eqref{e:stirling.correction.number} of $\bm{s}(\DD)$ gives
	\begin{align*}
	\bm{s}(\DD)
	-d_1(\DD)
	&=\bm{s}(\DD)
	\equiv\sum_{\bL} \sum_j (s_{\bL(j)}-1)
	-\sum_{\bT}\bigg(
	-x_{\bT}+\sum_{\bt\in\bT} (s_{\bt}-1-r_{\bt})
	\bigg)
	\\
	&=\sum_{\bL} \sum_j (s_{\bL(j)}-1)
	-\sum_{\bT} x_{\bT}\bigg(
		|\set{\bt : \bt \in\bT}|-1
		\bigg)\,,
	\end{align*}
where the last equality uses that for all $\bt\in\bT$ we have
$s_{\bt}=|\supp\starpi_{\bt}|=1+x_{\bT}+r_{\bt}$.
This matches the expression for $\wp(\DD)$ from \eqref{e:wp.of.R}, and concludes the proof.
\end{proof}
\end{lem}

\begin{lem}[{calculation of \eqref{e:dimension.count.pair.model}}]\label{l:dimension.count.two}
In the setting of Lemma~\ref{l:if.neg.def} we have
	\[d_2(\DD)
	=\sum_{\bL}
	\Bigg(
	(s_{\bL})^2-1-\sum_j \bigg[ (s_{\bL(j)})^2-1 \bigg]
	\Bigg)
	+\sum_{\bT}
	\Bigg(
	(s_{\bT})^2-1-\bigg[
	(x_{\bT})^2+2x_{\bT}+2x_{\bT}r_{\bT}+3r_{\bT}
	\bigg]
	\Bigg)
	\]
where $r_{\bT}$ denotes the sum of $r_{\bt}=\Ind{\set{\red,\blu}\subseteq\supp\starpi_{\bt}}$ over all edge types $\bt\in\bT$.

\begin{proof} 
Recall that $d_2(\DD)$ counts the dimension of $\set{\nu:\nu\sim\omega}$ where $\omega$ is the pair empirical measure. Since $\omega$ is fixed, similarly to \eqref{e:dimension.one.decompose.into.types} we can decompose
	\beq\label{e:dimension.two.decompose}
	d_2(\DD)
	=\sum_{\bL}d_2(\bL)+\sum_{\bT} d_2(\bT)\,.\eeq
Given a clause type $\bL$, let $A'$ be the matrix from \eqref{e:matrix.A.prime.clause}. Let $A[j]$ denote the submatrix of rows in $A'$ concerning the $j$-th edge incident to the clause, so that $A'$ consists of the submatrices $A[j]$ (for $1\le j\le k(\bL)$) together with the all-ones row. Let $A''$ be the matrix whose rows are given by the pairwise tensor products $u\otimes w$ of rows $u,w$ from $A'$, \bemph{except} if $u$ and $w$ come from distinct blocks $A[j]\ne A[k]$. Since $A'$ is of full rank (as was shown in the proof of Corollary~\ref{c:clause.strict.coherence.lagrange}), so is $A''$. It follows that
	\[
	d_2(\bL)
	=\dim\ker A''
	=(s_{\bL})^2-\rk A''
	=(s_{\bL})^2-1-\sum_j \bigg[ (s_{\bL(j)})^2-1\bigg]\,.
	\]
Similarly given a variable type $\bT$, let $B'$ be the matrix from \eqref{e:matrix.B.prime.variable}. Let $B^\red$ denote the submatrix of its first $r_{\bT}$ rows, which concern the edges around the variable that can take on both colors $\set{\red,\blu}$. Let $B''$ be the matrix whose rows are given by the pairwise tensor products $u\otimes w$ of rows $u,w$ from $B'$, \bemph{except} if $u$ and $w$ are distinct rows from $B^\red$. Since $B'$ is of full rank (as was shown in the proof of Lemma~\ref{l:dimension.count}), so is $B''$. It follows that
	\[d_2(\bT)
	=\dim\ker B''
	= (s_{\bT})^2-\rk B''
	= (s_{\bT})^2-\bigg[
	\Big(1+x_{\bT}+r_{\bT}\Big)^2-r_{\bT}(r_{\bT}-1)
	\bigg]\,.
	\]
Substituting into \eqref{e:dimension.two.decompose} proves the claim.
\end{proof}
\end{lem}

\begin{lem}[{proof of \eqref{e:dim.judicious.omega}}]\label{l:dimension.of.judicious.omega}
In the setting of Lemma~\ref{l:if.neg.def} we have
$j_2(\DD)=\bm{s}_2(\DD)- d_2(\DD)- 2\wp(\DD)$.

\begin{proof}
Recall that $j_2(\DD)$ counts the dimension of the space of judicious pair empirical measures $\omega$ around the product measure $\prodom$.
Given a pair empirical measure $\omega$ (which we assume throughout this proof to be close to $\prodom$), let $\pi_{\bt}$ be the average of the entries $\omega_{\bL,j}$ such that $\bL(j)=\bt$:
	\beq\label{e:pi.as.avg.of.omega.for.dim}
	\pi_{\bt}
	=\sum_{\bL}
	\pi_{\DD}(\bL\,|\,\bt)
	\omega_{\bL,j}
	\eeq
where $\pi_{\DD}(\bL\,|\,\bt)$ is as in \eqref{e:pi.DD.first.appearance}. If $\omega$ is judicious, then $\pi_{\bt}$ must also be judicious. In order for $\omega$ to be feasible, for each $\bT$ 
there must exist a measure $\mu_{\bT}$ on $\set{\minus,\plus,\free}^2$ which is consistent with $\pi_{\bt}$ for every $\bt\in\bT$. The measure $\mu_{\bT}$ must also be judicious, in the sense that its single-copy marginals are consistent with $\starpi_{\bt}$ for $\bt\in\bT$. Recall that $x_{\bT}\in\set{0,1,2}$ counts the number of frozen spins that can appear at a variable of type $\bT$. Let
	\[
	C' = \begin{pmatrix} C \\ \mathbf{1}^\st
	\end{pmatrix}
	\in\mathbb{R}^{(1+x_{\bT})\times(1+x_{\bT})}\,,
	\]
where for each frozen spin that can appear at $\bT$, the matrix $C$ contains a row which is the indicator of that spin. Then the matrix
	\[
	C''
	= \begin{pmatrix} 
	C\otimes \mathbf{1}^\st\\
	\mathbf{1}^\st\otimes C\\
	\mathbf{1}^\st\otimes \mathbf{1}^\st
	\end{pmatrix}
	\in\mathbb{R}^{
	(2x_{\bT}+1)
	\times (1+x_{\bT})^2}
	\]
encodes the judicious constraints on $\mu_{\bT}$ in the pair model. It follows that
	\[
	d_2(\mu_{\bT})
	\equiv\dim\bigg\{
	\mu_{\bT}
	:\textup{$\mu_{\bT}$ is judicious}\bigg\}
	=\dim\ker C''
	=(1+x_{\bT})^2
	-\bigg( 2x_{\bT}+1\bigg)
	=(x_{\bT})^2\,.
	\]
Now note that
for each $\bt\in\bT$, the choice of $\mu_{\bT}$ already fixes $\pi_{\bt}$ as a measure on $\set{\pur,\yel,\grn}^2$. It remains to separate \SPIN{purple} into \SPIN{red} and \SPIN{blue} in such a way that the resulting $\pi_{\bt}$ has single-copy marginals $\starpi_{\bt}$:
	\begin{align*}
	d_2(\pi_{\bt}\,|\,\mu_{\bT})
	&\equiv\dim\bigg\{
	\pi_{\bt}
	:\textup{$\pi_{\bt}$ is judicious
	and consistent with $\mu_{\bT}$}
	\bigg\}\\
	&=
	\dim\bigg\{
	\pi_{\bt}
	:\textup{$\pi_{\bt}$ is judicious}
	\bigg\}
	-\dim\bigg\{
	\mu_{\bT}
	:\textup{$\mu_{\bT}$ is judicious}\bigg\}
	=( s_{\bt}-1)^2
	-(x_{\bT})^2\,.
	\end{align*}
In order for $\omega$ to be consistent with $\pi_{\bt}$, it must satisfy the linear constraints 
\eqref{e:pi.as.avg.of.omega.for.dim}. Note that some of these are redundant with the judicious constraints on the individual measures $\omega_{\bL,j}$. Altogether
	\begin{align*}
	d_2( (\omega_{\bL,j})_{\bL(j)=\bt} \,|\,\pi_{\bt})
	&\equiv\dim\bigg\{
	(\omega_{\bL,j})_{\bL(j)=\bt}:
	\textup{$\omega$ is 
	judicious and consistent with $\pi_{\bt}$}\bigg\}\\
	&=\sum_{\bL,j : \bL(j)=t}
	( s_{\bL(j)}-1)^2
	- (s_{\bt}-1)^2
	\end{align*}
Combining these calculations gives
	\begin{align*}
	j_2(\DD)
	&= \sum_{\bT}\Bigg( d_2(\mu_{\bT})
	+\sum_{\bt\in\bT}
	\bigg[
	d_2(\pi_{\bt}\,|\,\mu_{\bT})
	+d_2( (\omega_{\bL,j})_{\bL(j)=\bt} \,|\,\pi_{\bt})
	\bigg]\Bigg)\\
	&=
	-\sum_{\bT} (x_{\bT})^2
	\bigg(|\set{\bt:\bt\in\bT}|-1\bigg)
	+\sum_{\bL,j} (s_{\bL(j)}-1)^2\,.
	\end{align*}
Now recall that 
$\bm{s}_2(\DD)$ is given in \eqref{e:stirling.correction.two},
and
$d_2(\DD)$ was calculated in Lemma~\ref{l:dimension.count.two}. 
Thus, if we abbreviate
$\wp_2(\DD)\equiv
\bm{s}_2(\DD)-d_2(\DD)-j_2(\DD)$, then we obtain
	\[
	\wp_2(\DD)
	=
	\sum_{\bT}\Bigg(
	(x_{\bT})^2 |\set{\bt:\bt\in\bT}|
	+ 2x_{\bT}+(2x_{\bT}+3)r_{\bT}
	-\sum_{\bt\in\bT}
	\Big[ (s_{\bt})^2-1\Big]
	\Bigg)
	+2\sum_{\bL,j} (s_{\bL(j)}-1)\,.
	\]
For $\bt\in\bT$, we can expand
$s_{\bt}=1 +x_{\bt} + r_{\bt}$, so
	\[
	\sum_{\bt\in\bT}
	\bigg( (s_{\bt})^2-1\bigg)
	=\sum_{\bt\in\bT}
	\bigg( (x_{\bT})^2 + 2x_{\bT} +2x_{\bT}r_{\bt}
		+ 3r_{\bt}\bigg)
	=\bigg((x_{\bT})^2 + 2x_{\bT}\bigg)
		|\set{\bt:\bt\in\bT}|
		+(2x_{\bT}+3)r_{\bT}\,,
	\]
where we recall that $r_{\bT}$ is the sum of $r_{\bt}$ over all $\bt\in\bT$. Substituting into the previous calculation gives
	\[
	\wp_2(\DD)
	= -2
	\sum_{\bT} x_{\bT}
	(|\set{\bt:\bt\in\bT}|-1)
	+2\sum_{\bL,j} (s_{\bL(j)}-1)\,,
	\]
which is exactly twice the value of $\wp(\DD)$ from \eqref{e:wp.of.R}. This proves the claim.
\end{proof}
\end{lem}

We conclude this section with the main applications of the result of Corollary~\ref{c:clause.strict.coherence.lagrange}:

\begin{cor}\label{c:cohere.weights} Let $U'$ be any finite bipartite tree, and suppose $U\subseteq U'$ is a tree containing at least one clause, whose leaves are all variables. For each edge $e=(av)$ in $U'$, let $\starpi_e$ be the canonical marginal on $e$ based on $B_r(v)$, the $r$-neighborhood of $v$ relative to $U'$. If all clauses in $U$ are strictly coherent, then there is a set of edge weights $\gm_e : \set{\RYGB}\to[0,\infty)$ (for $e$ in $U$, and with the convention $\gamma_e(\grn)\equiv1$), such that in the Gibbs measure on valid colorings of $U$
where the edges are weighted by $\gm$ and the boundary variables are weighted by $\dqstar$, 
all edge marginals agree with the canonical ones $\starpi$. (The Gibbs measure is explicitly given by \eqref{e:gamma.weighted.gibbs} below.)

\begin{proof} Let $\dqstar$ and $\hqstar$ be the canonical messages (based on $r$-neighborhoods) on the edges of $U$, as in Definition~\ref{d:canonical}. For each variable $v$ of $U$, let $\CCOLS_{\delta v}$ be the set of valid colorings $\usi_{\delta v}$ such that $\sigma_e\in\supp\starpi_e$ for all $e\in\delta v$. Now recall the discussion around \eqref{e:var.tuple.measure.weighted}: the probability measure
	\beq\label{e:within.proof.max.entropy.var}
	\nu_{\delta v}(\usi_{\delta v})
	= \f1{\dbz_v}
	\varphi_v(\usi_{\delta v})
	\prod_{e\in\delta v} \hqstar_e(\sigma_e)
	\eeq
is supported on valid colorings of $\delta v$, and has edge marginals on $\delta v$ that are consistent with $\starpi$. Moreover, it is clear from the above expression that if $e\in\delta v$ and $\starpi_e(\sigma)>0$, then we must have $\hqstar_e(\sigma)>0$. Consequently, for any $\usi_{\delta v}\in\CCOLS_{\delta v}$, we have $\varphi_v(\usi_{\delta v})=1$ and $\hqstar_e(\sigma_e)>0$ for all $e\in\delta a$, which means that \eqref{e:within.proof.max.entropy.var} is positive. That is to say, the measure $\nu_{\delta v}$ defined by \eqref{e:within.proof.max.entropy.var} has support exactly equal to $\CCOLS_{\delta v}$.

We next make a simple observation: for any edge $e$ of $U$ and any vertex $x$ incident to $e$, we have $\sigma\in\supp\starpi_e$ if and only if there exists
$\usi_{\delta x}\in\CCOLS_{\delta x}$ which has value $\sigma$ on edge $e$. The ``if'' direction is obvious.
For the ``only if'' direction, recall that we have constructed $\nu_{\delta x}$ with marginals $\starpi$ (this follows from Lemma~\ref{l:coher.feasibility} if $x$ is a clause, and from the above discussion if $x$ is a variable). Since $\starpi_e$ is the marginal of $\nu_{\delta x}$, if $\sigma\in\supp\starpi_e$ then the measure $\nu_{\delta x}$ must give positive weight to some $\usi_{\delta x}$ which has value $\sigma$ on edge $e$. Then $\usi_{\delta x}\in \supp \nu_{\delta x}\subseteq\CCOLS_{\delta x}$, as desired.

Now define $\CCOLS_U$ to be the set of valid colorings $\usi_U$ of $U$ such that $\sigma_e\in\supp\starpi_e$ for all edges $e$ of $U$. This is simply because $U$ is a tree, so we can construct an element of $\CCOLS_U$ as follows: start from any vertex $x$ of $U$, and choose $\usi_{\delta x}\in\CCOLS_{\delta x}$. For each vertex $y\in\pd x$, by the preceding observation, the set $\CCOLS_{\delta y}$ must contain an element $\usi_{\delta y}$ which agrees with $\usi_{\delta y}$ on the edge $e=(xy)$. Choose this $\usi_{\delta y}$, and proceed in the same way to color the edges incident to the neighbors of $y$, and so on until all of $U$ is colored. This shows $\CCOLS_U\ne\emptyset$.

Now recall from \eqref{e:defn.canonical.edge.marginal} that $\starpi_e(\sigma)$ is proportional to $\dqstar_e(\sigma)\hqstar_e(\sigma)$. In particular, if $\sigma\in\supp\starpi_e$, then $\dqstar_e(\sigma)$ must be positive, in which case we can define
	\[
	\gamma_e(\sigma)
	\equiv \f{w_e(\sigma)}{\dqstar_e(\sigma)}
	\dqstar_e(\grn)\,.
	\]
for $w_e$ as given by Corollary~\ref{c:clause.strict.coherence.lagrange}. 
 Recall that $\grn\in\supp\starpi_e$
and we took the convention
$w_e(\grn)\equiv1$, so now
$\gamma_e(\grn)\equiv1$ also. For $\sigma\notin\supp\starpi_e$ we simply define $\gamma_e(\sigma)=0$. We write $U=(V,F,E)$. Write $\Leaves U$ for all the leaves of $U$, which were assumed to be variables. For each $u\in\Leaves U$ let $a(u)$ denote the unique clause in $U$ that neighbors $u$. The Gibbs measure described in the statement of this lemma can be expressed as 
	\beq\label{e:gamma.weighted.gibbs}
	\nu(\usi_U)
	\cong
	\prod_{v\in V\setminus\Leaves U}
	\varphi_v(\usi_{\delta v})
	\prod_{a\in F}
	\hat{\varphi}_a(\usi_{\delta a})
	\prod_{e\in E}
	\gamma_e(\sigma_e)
	\prod_{u\in\Leaves U}
	\dqstar_{u a(u)}(\sigma_{u a(u)})
	\,,
	\eeq
where $\cong$ denotes the normalization constant, which is positive since $\supp\gamma_e=\supp\starpi_e$, $\supp\dqstar_e\supseteq\supp\starpi_e$, and $\CCOLS_U\ne\emptyset$. We claim that this measure $\nu$ has marginals $\starpi$ on the edges of $U$. To this end, note that the canonical messages $\dqstar,\hqstar$ solve the belief progagation equations for $\nu_U$:
\begin{enumerate}[(i)]
\item If $u\in\Leaves U$, then we regard $u$ as being weighted only by $\dqstar_{ua(u)}$, so then the \textsc{bp} message from $u$ to $a(u)$ will be $\dqstar_{ua(u)}$
simply by the standard conventions of \textsc{bp} at the leaves of trees. 
\item If $v$ is an internal variable of $U$, then we regard $v$ as being unweighted, and the equation
$\dqstar_{va}(\sigma)=\BP_{va}[\hqstar]$ is satisfied 
for all $a\in\pd v$ by definition of $\dqstar$ and $\hqstar$ (cf.\ Remark~\ref{r:first.rmk.coherence}). 
\item Finally, if $a$ is a clause of $U$, then we regard $a$ as being weighted by the product of $\gamma_e$ over $e\in\delta a$. Consider
	\beq\label{e:putative.nu.delta.a}
	\nu_{\delta a}(\usi_{\delta a})
	\cong
	\hat{\varphi}_a(\usi_{\delta a})
	\prod_{e\in \delta a}
	\bigg\{
	\gamma_e(\sigma_e)
	\dqstar_e(\sigma_e)
	\bigg\}
	\cong \hat{\varphi}_a(\usi_{\delta a})
	\prod_{e\in \delta a}
	w_e(\sigma_e)\,.\eeq
(We are not yet claiming that $\nu_{\delta a}$ is the marginal of $\nu$.) By the result of Corollary~\ref{c:clause.strict.coherence.lagrange}, the measure $\nu_{\delta a}$ has marginals $\starpi$, so
for all $e=(av)\in\delta a$ we have
	\[
	\starpi_e(\sigma) \cong
	\sum_{\usi_{\delta a}}
	\Ind{\sigma_e=\sigma}
	\hat{\varphi}_a(\usi_{\delta a})
	\prod_{e'\in \delta a}
	\bigg\{
	\gamma_{e'}(\sigma_{e'})
	\dqstar_{e'}(\sigma_{e'})
	\bigg\}
	\cong
	\dqstar_e(\sigma)
	\cdot \BP_{av}[\dqstar;\gamma](\sigma)\,.
	\]
On the other hand we know $\starpi_e(\sigma)\cong\dqstar_e(\sigma)\hqstar_e(\sigma)$, so it must be that
$\BP_{av}[\dqstar;\gamma](\sigma)=\hqstar_{av}(\sigma)$.
\end{enumerate}
This proves our claim that $(\dqstar,\hqstar)$ solves
the belief propagation equations for $\nu$. It follows from this that the marginal of $\nu$ on each $\delta a$ is indeed correctly described by \eqref{e:putative.nu.delta.a}, which in turn implies that the edge marginals are consistent with $\starpi$.
\end{proof}
\end{cor}

In \S\ref{ss:explicit.lagrange.multipliers} we will show that if the clauses are nice (Definition~\ref{d:nice}), then we can estimate the weights of Corollary~\ref{c:cohere.weights}, which will be needed in the second half of the paper. For the precise statement see 
Corollary~\ref{c:clause.bp.weights.explicit}.

\begin{cor}\label{c:clause.bp.weights}
In the setting of Corollary~\ref{c:cohere.weights}, 
there is a system $\Lmstar_U$ of positive variable weights
	\[
	\Lm_v = \begin{pmatrix}
	\lm_v \equiv (\lm_v(\plus)=1, \lm_v(\minus),\lm_v(\free))\\
	\lm_e \equiv \lm_e(\red)
	\end{pmatrix}
	\]
such that the $\Lmstar_U$-weighted Gibbs measure on valid colorings of $U$ has all edge marginals agreeing with the canonical edge marginals $\starpi$.
(However, since the weights were shifted from clauses to variables, the associated \textsc{bp} messages will no longer be $\dqstar,\hqstar$.)

\begin{proof} 
The Gibbs measure \eqref{e:gamma.weighted.gibbs}
is equivalent to the Gibbs measure in which all variables are unweighted;
all internal edges $e$ of $U$ are weighted by $\tilde{\gamma}_e=\gamma_e$; and all leaf edges $e$ of $U$ are weighted by
	\[\tilde{\gamma}_e(\sigma)
	= \f{\gamma_e(\sigma) \dqstar_e(\sigma)}
		{\sum_\tau
		\gamma_e(\tau) \dqstar_e(\tau)}\,.
	\]
To define $\Lmstar_U$, we redistribute the weights as follows. For each internal variable $v$ of $U$,
we let $x_v\in\set{\minus,\plus,\free}$ denote the frozen spin corresponding corresponding to $\usi_{\delta v}$, and put the weight
	\[
	\Lmstar_v(\usi_{\delta v})
	\equiv
	\lmstar_v(x_v)
	\prod_{e\in\delta v}
	\lmstar_e(\sigma_e)\,,
	\]
where $\lmstar_v$ and $\lmstar_e$ are defined by 
\eqref{e:redistribute.weights}. For each leaf variable $v\in\Leaves U$, the set $\delta v$ consists of a single edge $e$, and in this case we will simply take the weight
	\[
	\lmstar_v(\sigma_e)
	\equiv\lmstar_e(\sigma_e)
	\equiv\tilde{\gamma}_e(\sigma_e)\,.
	\]
For the redistributed weights,
the \textsc{bp} messages on all edges $e$ of $U$ are given by
	\beq\label{e:redistributed.bp.messages}
	\dqbul_e(\sigma_e)
	\cong \dqstar_e(\sigma_e)
	\gamma_e(\sigma_e)\,,\quad
	\hqbul_e(\sigma_e)
	\cong \f{\hqstar_e(\sigma_e)}
		{\gamma_e(\sigma_e)}\,.
	\eeq
Thus the messages for the redistributed weights
$\Lmstar_U$ also satisfy the familiar identity
$\starpi_e(\sigma_e)
\cong \dqbul_e(\sigma_e)\hqbul_e(\sigma_e)$.
Moreover, if $e$ is a leaf edge then
$\lmstar_e=\tilde{\gamma}_e=\dqbul_e$.
\end{proof}
\end{cor}

\label{proofoverview}

\noindent\bemph{Outline of remaining sections.} The remainder of this paper is organized as follows:

\begin{enumerate}[--]
\item In Section~\ref{s:onersb} we analyze the distributional $\onersb$ recursion \eqref{e:intro.dist.recurs} to prove
Propositions~\ref{p:fp}~and~\ref{p:ubd}.

\item In Section~\ref{s:processing} we analyze the random $k$-\textsc{sat} graph and the preprocessing algorithm to prove Propositions~\ref{p:small.fraction.removed.in.processing}--\ref{p:unif}.

\item In Section~\ref{s:sep} we prove Propositions~\ref{p:ext}~and~\ref{p:sep}.
\item Sections~\ref{s:contract}--\ref{s:burnin} are devoted to the proof of Proposition~\ref{p:second.moment.judicious}.

\item In Section~\ref{s:monotonicity} we prove Proposition~\ref{p:phi}.
\end{enumerate}

\section{One-step RSB threshold}\label{s:onersb}

In this section we analyze the distributional $\onersb$ recursion \eqref{e:intro.dist.recurs} to prove Propositions~\ref{p:fp}~and~\ref{p:ubd}. The section is organized as follows:
\begin{enumerate}[--]
\item In \S\ref{ss:dist.rec} we formally define the random Galton--Watson tree that arises as a local weak limit of the $\ksat$ graph. We prove a basic estimate, Lemma~\ref{l:martingale.bound.PGW}, on the volume of neighborhoods in the random tree. We also formalize the natural coupling between the $\onersb$ recursion and the random tree.

\item In \S\ref{ss:technical.bounds.on.dist.rec} 
we prove preliminary concentration bounds on the effect of the distributional recursion.

\item In \S\ref{ss:root.is.robust} we use the bounds from \S\ref{ss:technical.bounds.on.dist.rec} to show that the root of the Galton--Watson tree is very likely to be \bemph{nice} in a strong sense
(Proposition~\ref{p:uPGW.is.J.robust}). This will be used in \S\ref{ss:bootstrap.defects} to bound the occurrence of \bemph{defective} regions.

\item In \S\ref{ss:pgw.stability} we use the coupling
of the distributional recursion with the tree to show
Proposition~\ref{p:one.stable}, which says 
 that the root of the Galton--Watson tree is very likely to be \bemph{stable}. From this we will deduce the result of Proposition~\ref{p:fp}.
Proposition~\ref{p:one.stable} will furthermore be used
in the proof of Proposition~\ref{p:first.moment.exponent},
and also in \S\ref{ss:preprocessing.probabilistic} to assure that the initial set $A$ of preprocessing (Definition~\ref{d:proc}) is small.

\item In \S\ref{ss:threshold.ubd} we prove 
 Proposition~\ref{p:ubd}, the $\onersb$ upper bound on the satisfiability threshold. This is deduced from the interpolation bounds of \cite{FrLe:03,MR2095932}.

\item Lastly, in \S\ref{ss:judicious.first.mmt} we prove Proposition~\ref{p:first.moment.exponent}, showing that the first moment of judicious colorings 
asymptotically is lower bounded (in the exponent) by the $\onersb$ free energy $\Phi(\alpha)$ of \eqref{e:phi.alpha}.
\end{enumerate}
It is assumed throughout the section, even when not explicitly stated, that $k\ge k_0$ and $\alpha$ satisfies \eqref{e:alpha.regime}.

\subsection{Random trees and the distributional recursion}
\label{ss:dist.rec}

We already mentioned in \S\ref{ss:canonical} 
that the random $\ksat$ graph converges
``locally in distribution'' (also termed ``locally in law,'' or ``locally weakly'') to the Galton--Watson measure $\uPGW\equiv\uPGW^\alpha$. We begin this section by formally restating the definition of $\uPGW$, and then reviewing some basic notions of local weak convergence which will be used later.



\begin{dfn}[Galton--Watson measure, $\uPGW$]\label{d:uPGW} We define the \bemph{bipartite Poisson Galton--Watson tree} 
to be the random bipartite factor tree $\tree$ which is generated as follows: starting from a root variable $\vrt$, each variable independently generates $\Pois(\alpha k)$ child clauses, and each clause generates $k-1$ child variables. Each edge is labelled with a literal $\lit$ which takes values $\plus$ or $\minus$ with equal probability, independently over all the edges. We write $\uPGW$ for the law of this tree.
\end{dfn}

\begin{rmk}\label{r:unimodular}
The measure $\uPGW$ is the ``local weak limit'' of the random $\ksat$ graph in the following formal sense: if $\GG_n=(V_n,F_n,E_n)$ denotes an instance of random $\ksat$ at clause density $\alpha$ on $n$ variables,
and $I_n$ is a uniformly random element of $V_n\equiv[n]$,
then it holds for any finite $r$ that
	\[
	\lim_{n\to\infty} \P\bigg(B_r(I_n;\GG_n)\cong T\bigg)
	= \uPGW^\alpha\bigg( B_r(\vrt;\tree) \cong T \bigg)
	\]
for every $T$. In the above, $B_r(I_n;\GG_n)$ is the $r$-neighborhood of $I_n$ in $\GG_n$, viewed as a graph rooted at $I_n$. Likewise $B_r(\vrt;\tree)$
is the $r$-neighborhood of $\vrt$ in $\tree$,
viewed as a graph rooted at $\vrt$. Finally $T$ is any graph with a root $o$, and $\cong$ denotes isomorphism of rooted graphs. The measure $\uPGW$ is \bemph{unimodular} (cf.\ \cite{MR1336708} and \cite[Defn.~2.1]{MR2354165}):
	\beq\label{e:unimodular}
	\int\bigg[\sum_u
	f(\tree,\vrt,u)\bigg]
	d\uPGW(\tree)
	=\int\bigg[\sum_u
	f(\tree,u,\vrt)\bigg]
	d\uPGW(\tree)
	\eeq
where the sum goes over all variables $u$ in $\tree$, and $f$ is any nonnegative Borel function on the space $\mathcal{G}_{\star\star}$ of bi-rooted graphs.\footnote{A bi-rooted graph is a graph rooted at an ordered pair of vertices. For a detailed account and careful discussion of topological considerations (in particular, what it means to be a Borel function on the space $\mathcal{G}_{\star\star}$), we refer to \cite{MR2354165}.} To see why \eqref{e:unimodular} should hold, note that changing the order of summation gives
	\beq\label{e:unim.n}
	\E\sum_{u\in \pd I_n} f(\GG_n,I_n ,u)
	=\f1{|V_n|}
	\sum_{v\in V_n}\bigg[
	\sum_{u\in N(v)}
	f(\GG_n , v,u )\bigg]
	=\E\sum_{u\in \pd I_n} f(\GG_n,u,I_n )\,.
	\eeq
In the limit $n\to\infty$, the left-hand side of \eqref{e:unim.n} converges to the left-hand side of
\eqref{e:unimodular}, while the right-hand side of 
\eqref{e:unim.n} converges to the right-hand side of
\eqref{e:unimodular}.
\end{rmk}

We emphasize that under the measure $\uPGW$, the root $\vrt$ plays a special role: in the local weak limit interpretation, $\vrt$ represents a uniformly random variable, while other vertices in the tree represent random \emph{neighbors}. Thus the root degree $|\pd\vrt|$ is distributed as $\Pois(\alpha k)$, while for variables $u\ne\vrt$, the degree $|\pd u|$ is distributed as a \bemph{size-biased} $\Pois(\alpha k)$, with probability mass function
	\[\overline{\POIS}_{\alpha k}(j)
	=\f{j\POIS_{\alpha k}(j) }{ \alpha k}\,.\]
The most obvious distinguishing feature of $\vrt$ is that it has degree zero with positive probability, while any other vertex of the random tree must have positive degree since it connects to its parent. In the above description of the $\uPGW$ law, we used the fact that
	\[\overline{\POIS}_{\alpha k}(j)
	= 
	\f{e^{-\alpha k} (\alpha k)^j\cdot j}{j!
		\cdot \alpha k}
	= \f{e^{-\alpha k} (\alpha k)^{j-1}}{(j-1)!}
	= \POIS_{\alpha k}(j-1)\,,\]
meaning that a size-biased $\Pois(\alpha k)$ random variable is equidistributed as a $[1 + \Pois(\alpha k)]$ random variable --- this explains why variables $u\ne\vrt$ generate a $\Pois(\alpha k)$ number of children. To understand how (canonical) messages behave in the random $\ksat$ graph, we study how they behave in the limiting random tree $\tree\sim\uPGW$. For this purpose, we make the following definitions:


\begin{dfn}[variable-to-clause measure, $\PGW$]\label{d:vtoc.PGW}
Let $\tree\sim\uPGW$ with root variable $\vrt$. Let $\tree_{\vrt\crt}$ be the tree $\tree$ together with one additional edge $\ert\equiv(\vrt\crt)$ incident to the root, equipped with a random sign $\lit_{\ert}$. We think of $\ert$ as the parent edge of $\vrt$, pointing to a deleted clause $\crt$. Then $\tree_{\vrt\crt}$ is a random variable-to-clause tree, and we denote its law $\PGW$.\end{dfn}

\begin{dfn}[clause-to-variable measure, $\hPGW$]\label{d:ctov.PGW} Let $\tree_{\crt\vrt}$ be the random tree formed as follows: start with a root clause $\crt$, and attach $k-1$ independent samples of $\PGW$ (Definition~\ref{d:vtoc.PGW}) as subtrees to $\crt$. Then attach to $\crt$ one additional incident edge $\ert$, which we think of as the parent edge of $\crt$, pointing to a deleted variable $\vrt$. The result is a random clause-to-variable tree $\tree_{\crt\vrt}$, whose law we denote as $\hPGW$.\end{dfn}

The definitions of $\PGW$ and $\hPGW$ can be viewed in this way: let $\tree\sim\uPGW$, $a\in\pd\vrt$, and $u\in\pd a\setminus\vrt$. Recall from Definition~\ref{d:frozen.model.bdy.conditions} and the subsequent discussion that the message from $u$ to $a$ is defined in terms of $\tree_{ua}$ --- the component of $\tree\setminus a$ containing $u$, including the edge $(ua)$ but not including $a$ itself. It follows from the above discussion that if we also remove $(ua)$ from $\tree_{ua}$, then the resulting tree (rooted at $u$) is equidistributed as the original tree $\tree\sim\uPGW$. It follows that the law of $\tree_{ua}$ is precisely the measure $\PGW$. Similarly, if $\tree\sim\uPGW$ and $a\in\pd\vrt$, then the law of $\tree_{a\vrt}$ is given by $\hPGW$. In summary, in order to understand messages in the random $\ksat$ graph, we will study messages from $\vrt$ to $\crt$ under the law $\PGW$, and messages from $\crt$ to $\vrt$ under the law $\hPGW$. In fact, for technical reasons
(which will become apparent in the proofs of this section and the next one) we will study messages on slightly more general trees, defined as follows:

\begin{dfn}[Galton--Watson based on fixed tree, $\uPGW(T)$]
\label{d:PGW.based.on.T}
Let $T$ be any fixed tree rooted at a variable $\vrt$,
such that all clauses in $T$ have at most $k-1$ children.
At every variable $v$ of $T$ (including $v=\vrt$), attach an independent subtree $\tree_v\sim\uPGW$ (by identifying $v$ with the root of $\tree_v$). At every clause $a$ of $T$, if $a$ has $k-1-j$ children, attach $j$ independent samples of $\PGW$ as subtrees to $a$. Let $\uPGW(T)$ be the law of the resulting tree.

Similarly, if $T_{\vrt\crt}$ is a fixed variable-to-clause tree where every clause has at most $k-1$ children, we can perform the same procedure as above to obtain a random variable-to-clause tree, whose law we denote $\PGW(T_{\vrt\crt})$. Likewise, if $T_{\crt\vrt}$ is any fixed clause-to-variable tree where every clause has at most $k-1$ children, we use $\hPGW(T_{\crt\vrt})$ to denote the law of the random clause-to-variable tree based on $T_{\crt\vrt}$.
\end{dfn}

\begin{dfn}[Galton--Watson with random deletions,
$\uPGW_\EPS$]\label{d:PGW.with.clauses.one.less}
Let $\uPGW_\EPS$ be the law of the random tree generated as follows: starting from a root variable $\vrt$, each variable independently generates $\Pois(\alpha k)$ clauses, and each clause independently generates either $k-1$ child variables (with probability $1-\EPS$)
or $k-2$ child variables (with probability $\EPS$). When $\EPS=0$, the measure $\uPGW_\EPS$ coincides with the measure $\uPGW$ of Definition~\ref{d:uPGW}. The measure $\uPGW_\EPS$ is also unimodular, because it can also be obtained via a local weak limit: let $\hat{\EPS}$ and $\hat{\alpha}$ be defined by
	\[
	\EPS = \f{(k-1)\hat{\EPS}}{k-\hat{\EPS}}\,,\quad
	\hat{\alpha} = \f{\alpha}{1-\hat{\EPS}/k}
	\]
Start with $n$ isolated variables and $n\hat{\alpha}$ isolated clauses. For each of the first $1-\hat{\EPS}$ fraction of the clauses, put $k$ edges to randomly chosen variables. For the last $\hat{\EPS}$ fraction of the clauses, put $k-1$ edges to randomly chosen variables. The total number of edges in the resulting random graph is
	\[
	n\hat{\alpha} \bigg((1-\hat{\EPS})k
	+\hat{\EPS}(k-1)\bigg)
	= n\alpha k\,.
	\]
In the limit $n\to\infty$, this random graph converges locally weakly to the $\uPGW_\EPS$ measure, justifying our claims that $\uPGW_\EPS$ is unimodular. Define likewise $\PGW_\EPS$ and $\hPGW_\EPS$ to be the obvious generalizations of the measures $\PGW$ and $\hPGW$ from Definitions~\ref{d:vtoc.PGW}~and~\ref{d:ctov.PGW}. 
We hereafter write $\clausedeg_\EPS$ for the probability measure which puts weight $1-\EPS$ on $k$, and weight $\EPS$ on $k-1$.
\end{dfn}

The next lemma gives a simple bound on the growth of balls under the $\uPGW$ measure:

\begin{lem}\label{l:martingale.bound.PGW}
Let $\tree\sim\uPGW$, and consider the $\ell$-neighborhood of the root, $B_\ell(\vrt)\equiv B_\ell(\vrt;\tree)$. We have
	\beq\label{e:martingale.bound.PGW}
	e_\ell \equiv \int \exp\bigg\{
		\f{|B_\ell(\vrt;\tree)|}{(\alpha k^2)^\ell}
		\bigg\} \,d\uPGW(\tree) \le e
	\eeq
for all $\ell\ge0$.

\begin{proof}
Let $S_\ell$ be the number of variables in $\tree$ at distance exactly $\ell$ from the root $\vrt$, with $S_0=1$. Then
	\[e_\ell = 
	\E\Bigg[ \exp\Bigg\{ \f1{(\alpha k^2)^\ell}
		\sum_{j=0}^\ell S_j\Bigg\}\Bigg]\,.
	\]
Let $\mathscr{F}_\ell$ denote the $\sigma$-field generated by $B_{\ell-1}(\vrt)$: then
	\[
	\E\Bigg[\exp\Bigg\{ 
	\f{S_\ell }{(\alpha k^2)^\ell} 
	\Bigg\}\,\Bigg|\,
		\mathscr{F}_{\ell-1}\Bigg]
	=\exp\Bigg\{
	S_{\ell-1}\Bigg[
	\exp \bigg( \f{k-1}{(\alpha k^2)^\ell} \bigg) -1
	\Bigg] \alpha k
	\Bigg\}
	\le 
	\exp\Bigg\{ \f{S_{\ell-1}
		[1-\Theta(1/k)]}{(\alpha k^2)^{\ell-1}}
	\Bigg\}\,.
	\]
It follows by iterated expectations that
	\[
	e_\ell
	\le
	\E\Bigg[ \exp\Bigg\{ \f1{(\alpha k^2)^\ell}
		\sum_{j=0}^{\ell-2} S_j
		+ \f{S_{\ell-1}}{(\alpha k^2)^\ell}
		+ \f{S_{\ell-1}
		[1-\Theta(1/k)]}{(\alpha k^2)^{\ell-1}}
		\Bigg\}
		\Bigg]
	\le e_{\ell-1} \le e_0 = e\,.
	\]
This proves the claim.
\end{proof}
\end{lem}

We now review the distributional recursion introduced in \S\ref{ss:intro.onersb.threshold}, and which we saw in \S\ref{ss:wp.recursions} is related to the frozen model
on finite trees. (Recall, in particular, the
similarity between \eqref{e:intro.dist.recurs} and \eqref{e:second.introduction.of.Rec}.) Let $d^\plus,d^\minus$ be independent samples from the $\Pois(\alpha k/2)$ distribution, and write $\ud \equiv (d^\plus,d^\minus)$. We denote their probability mass function by
	\[\POpm
	(\vec d)
	\equiv\f{e^{-\alpha k}
	(\alpha k/2)^{d^\plus + d^\minus}}
		{(d^\plus)!(d^\minus)! }\,.\]
Note that if $\tree_{\vrt\crt}\sim\PGW$ as discussed above, then (with the notation of \eqref{e:incident.edges.of.same.sign}) the pair
	\[\bigg(\Big|\delta v(\plus\crt)\Big|,
		\Big|\delta v(\minus\crt)\Big|\bigg)\]
has law exactly $\POpm$.
We next set some notations for laws on messages:
\begin{enumerate}[--]
\item \textbf{Variable-to-clause messages.} As before, we use $\bmeta$ to denote a probability measure over $\set{\plus,\minus,\free}$, interpreted as a message from some variable $v$ to one of its neighboring clauses $a$. We will often summarize $\bmeta$ by the scalar value $\eta=\bmeta(\minus)\in[0,1)$, which is interpreted as the chance (according to the message) for $v$ not to satisfy $a$. We write $\cP$ for the space of probability measures over $\bmeta$, and use $\bmu$ to denote elements of $\cP$. We write $\PINT$ for the space of probability measures over $\eta\in[0,1)$, and use $\mu$ to denote elements of $\PINT$. The mapping from $\bmeta$ to $\eta=\bmeta(\minus)$ naturally induces a mapping from $\bmu\in\cP$ to $\mu\in\PINT$.
\item \textbf{Clause-to-variable messages.} Likewise, we use $\bhu$ to denote a probability measure over $\set{\plus,\free}$, with the interpretation of a message from some clause $a$ to one of its neighboring variables $v$. We will often summarize $\bhu$ by the scalar value $\hat{u}=\bhu(\plus)\in[0,1)$, which is interpreted as the chance (according to the message) that $a$ is forcing to $v$. Of course, since $\bhu$ is a probability measure over only two elements, the correspondence between $\bhu$ and $\hat{u}$ is a bijection. We write $\hcP$ for the space of probability measures over $\bhu$, and use $\bhmu$ to denote elements of $\hcP$. The (one-to-one) mapping from $\bhu$ to $\hat{u}=\bhu(\plus)$ naturally induces a one-to-one mapping from $\bhmu\in\hcP$ to $\hat{\mu}\in\PINT$.
\end{enumerate}
It may be useful to keep in mind that, ultimately, the measures $\bmu$ and $\bhmu$ will capture the influence of the random local geometry in the graph. To construct $\bmu$ and $\bhmu$ appropriately, we next define the mappings that will capture the distributional effect of the frozen model recursions \eqref{e:first.defn.of.bhu}~and~\eqref{e:intro.defn.bmeta}.


\begin{dfn}[distributional effect of clause recursion~\eqref{e:first.defn.of.bhu}]
\label{d:distributional.clause.recursion}
Given $\mu\in\PINT$, let $\ueta'\equiv(\eta_j)_{j\ge1}$ denote a sequence of i.i.d.\ samples from $\mu$. Independently, let $K\sim\clausedeg_\EPS$ (as in Definition~\ref{d:PGW.with.clauses.one.less}). Let $\bhu\equiv\bhu(\ueta')$ be the probability measure on $\set{\plus,\free}$ defined by
	\[
	\Big(\bhu(\plus),\bhu(\free)\Big)
	=\bigg(
	\prod_{j=1}^{K-1}\eta_j,
	1-\prod_{j=1}^{K-1}\eta_j\bigg)\,.
	\]
Let $\hREC_\EPS\mu$ denote the law of $\bhu$.
Thus $\hREC_\EPS$ defines a mapping from $\PINT$ to $\hcP$, which captures how randomness is passed through the clause update \eqref{e:first.defn.of.bhu} of the frozen model recursions. We write $\hREC\equiv\hREC_0$ for the $\EPS=0$ case.
\end{dfn}

\begin{dfn}[distributional effect of variable recursion~\eqref{e:intro.defn.bmeta}]
\label{d:distributional.var.recurs}
 Given $\bhmu\in\hcP$, let $\vec{\bhu}\equiv(\bhu^\plus_i,\bhu^\minus_i)_{i\ge1}$
be an array of i.i.d.\ samples from $\bhmu$, and 
write $\hat{u}^\PM_i\equiv\bhu^\PM_i(\plus)\in[0,1)$.
Let $\ud\sim\POpm$, and define
	\[
	\Pi^\plus\equiv
	\Pi^\plus(\ud,\vec{\hat{u}})
	\equiv
	\prod_{i=1}^{d^\plus}\bigg(
	1 - \hat{u}^\plus_i
	\bigg),\quad
	\Pi^\minus\equiv
	\Pi^\minus(\ud,\vec{\hat{u}})\equiv
	\prod_{i=1}^{d^\plus}\bigg(
	1 - \hat{u}^\minus_i
	\bigg)\,,
	\]
and note that $\Pi^\PM\in(0,1]$. Let $\bmeta\equiv\bmeta(\vec{\bhu})$ be the probability measure defined by
	\beq\label{e:eta.in.terms.of.signed.Pis}
	\bigg(
	\bmeta(\plus),\bmeta(\minus),\bmeta(\free)
	\bigg)
	=\bigg(\f{\Pi^\minus(1-\Pi^\plus)}
		{ \Pi^\plus + \Pi^\minus - \Pi^\plus \Pi^\minus },
	\f{\Pi^\plus(1-\Pi^\minus)}
		{\Pi^\plus + \Pi^\minus - \Pi^\plus \Pi^\minus },
	\f{\Pi^\plus\Pi^\minus}
		{\Pi^\plus + \Pi^\minus - \Pi^\plus \Pi^\minus }
	\bigg)\,.
	\eeq
Let $\dREC\bhu$ denote the law of this $\bmeta$. Thus $\dREC$ defines a mapping from $\hcP$ to $\cP$, which captures how randomness is passed through the variable \textsc{bp} recursion.
\end{dfn}

\begin{dfn}[full distributional recursion] \label{d:full.dist.recurs}
Let $\REC_\EPS$ denote the composition $\dREC \circ \hREC_\EPS$: this gives a mapping from $\PINT$ to $\cP$, which captures how randomness is passed through one full update (clause updates follows by variable update) of variable-to-clause messages. We can explicitly describe $\REC_\EPS$ as follows: given $\mu\in\PINT$, let $\ueta$ be
an array of i.i.d.\ samples from $\mu$, as in \eqref{e:defn.array.ueta}. Independently, let $\ud\sim\POpm$, and let $\vec{K}\equiv(K_i)_{i\ge1}$ be a sequence of i.i.d.\ samples from $\clausedeg_\EPS$ (as in Definition~\ref{d:PGW.with.clauses.one.less}). 
From this, define the random variables
	\[\Pi^\PM
	\equiv \Pi^\PM(\ud,\ueta)
	\equiv \prod_{i=1}^{d^\PM}
			\bigg(
			1 - \prod_{j=1}^{K_i-1} \eta^\PM_{ij}
			\bigg)
	=\prod_{i=1}^{d^\PM}
			\bigg(
			1 - \hat{u}^\PM_i
			\bigg)
	\,.\]
Substitute these $\Pi^\PM$ into \eqref{e:eta.in.terms.of.signed.Pis}
to define $\bmeta=\bmeta(\ud,\vec{K},\ueta)$. Then the law of $\bmeta(\ud,\vec{K},\ueta)$ is $\REC_\EPS\mu\in\cP$. As in \eqref{e:intro.dist.recurs} we let
	\[R(\ud,\vec{K},\ueta) 
	\equiv \Big[\bmeta(\ud,\vec{K},\ueta)
	\Big](\minus)\,.\]
As we noted in \S\ref{ss:intro.onersb.threshold},
since all the $\eta$'s must lie in $[0,1)$,
we must have $\Pi^\PM\in(0,1]$ and therefore $R(\ud,\vec{K},\ueta)\in[0,1)$. We finally define $\Rec_\EPS\mu\in\PINT$ to be the law of $R(\ud,\vec{K},\ueta)$, so $\Rec_\EPS$ gives a mapping from $\PINT$ to itself.\footnote{The formal characterization of $\Rec_\EPS$ is that for any measurable $B\subseteq [0,1)$,
	\[(\Rec_\EPS\mu)( \eta \in B )= 
	\sum_{\ud}
	\POpm(\vec d) 
	\int \Bigg[
	\int\Ind{R( \ud,\vec{K},\ueta )\in B}
	\, d\mu^{\otimes}(\ueta)\Bigg]
	\, d(\clausedeg_\EPS)^\otimes(\vec{K})
	\,,\]
where we abbreviate $\mu^{\otimes}$ for the law of the array $\ueta$ of i.i.d.\ samples from $\mu$, and similarly we abbreviate $(\clausedeg_\EPS)^\otimes$
for the law of the sequence $\vec{K}$ of i.i.d.\ samples from $\clausedeg_\EPS$.} We denote $\REC\equiv\REC_\EPS$ and $\Rec\equiv\Rec_\EPS$ for the $\EPS=0$ case.
\end{dfn}

\begin{dfn}[coupled sequences of messages] 
\label{d:coupled.seq.of.messages}
Recalling the statement of Proposition~\ref{p:fp}, let $\mu^0=\delta_{1/2}$. For $\ell\ge1$ let $\mu^{\ell,\EPS}\equiv(\Rec_\EPS)^\ell\mu^0$ and $\bmu^{\ell,\EPS}\equiv\REC_\EPS\mu^{\ell-1,\EPS}$. Moreover, for all $\ell\ge0$ let $\bhmu^{\ell+1/2,\EPS}\equiv\hREC_\EPS\mu^{\ell,\EPS}$. The natural mapping from $\cP$ to $\PINT$ (as discussed above) takes $\bmu^{\ell,\EPS}$ to $\mu^{\ell,\EPS}$. Recalling the discussion below \eqref{e:def.FF.of.tree}, if we take $\tree_{\vrt\crt} \sim\PGW_\EPS$ and define the random sequence
	\beq\label{e:coupling.of.all.eta.ell}
	(\bmeta^\ell)_{\ell\ge0}
	\equiv \bigg(\FF_\ell(
	\tree_{\vrt\crt} )\bigg)_{\ell\ge0}\,,
	\eeq
then the marginal law of each $\bmeta^\ell$ is precisely $\bmu^{\ell,\EPS}$. The marginal law of $\eta^\ell \equiv \bmeta^\ell(\minus) \in [0,1)$ is 
$\mu^{\ell,\EPS}$. Likewise, if we take $\hat{\tree}_{\crt\vrt}\sim\hPGW_\EPS$ and define
	\[
	(\bhu^{\ell+1/2})_{\ell\ge0}
	\equiv\bigg(
	\FF_{\ell+1/2}(\hat{\tree}_{\crt\vrt})
	\bigg)_{\ell\ge0}\,,
	\]
then the marginal law of each $\bhu^{\ell+1/2}$ is precisely $\bhmu^{\ell+1/2,\EPS}$. We drop $\EPS$ from the notation when $\EPS=0$.
\end{dfn}

\subsection{Concentration bounds for the distributional recursion}
\label{ss:technical.bounds.on.dist.rec}

In this subsection we prove two lemmas on the distributional recursion described above: Lemma~\ref{l:X.tail.bounds} studies the distributional effect of a single clause update, and Lemma~\ref{l:mu.ell.tail.bounds} gives concentration bounds on the variable-to-clause messages. The lemmas will be applied below in \S\ref{ss:root.is.robust} to show that under the $\uPGW$ measure, the root is $1$-nice with very good probability.

\begin{lem}\label{l:mu.ell.tail.bounds}
Let $\mu$ be any probability measure over $\eta\in[0,1)$ that satisfies the bounds
\begin{enumerate}[(I)]
\item \label{i:mu.ell.upper.tail.bound} 
$\mu(\eta \ge 1/2+s)
\le \mu(\log[\eta/(1-\eta)] \ge 4 s)\le \exp(-9s2^{k/4})$ for all $s\ge 2^{-k/4}$;
\item \label{i:mu.ell.left.tail.bound}
$\mu(\eta\le 1/2-s) \le \exp(-9s2^{k/4})$
 for all $2^{-k/4} \le s \le 1/2$.
\end{enumerate}
Then the measure $\mu^{\ell,\EPS}\equiv(\Rec_\EPS)^\ell\mu$ satisfies the same bounds
\eqref{i:mu.ell.upper.tail.bound}~and~\eqref{i:mu.ell.left.tail.bound} for all $\ell\ge0$. In addition, for all $\ell\ge1$, the measure $\bmu^{\ell,\EPS}=\REC_\EPS\mu^{\ell-1,\EPS}$ satisfies the bound
	\beq\label{i:mu.ell.eta.of.free.estimate}
	\bmu^{\ell,\EPS}\bigg(
	\Big|2^{k+1}\bmeta(\free)-1\Big|
	\ge \f1{2^{k/4}} \bigg)
	\le \f1{\exp(2^{k/4})}\,.
	\eeq
The estimates of this lemma hold for all $0\le\EPS\le1$.
\end{lem}


\begin{rmk*} In Lemma~\ref{l:mu.ell.tail.bounds}, we remark that the right tail bound \eqref{i:mu.ell.upper.tail.bound} tends to zero in the limit $s\to\infty$, which is consistent with $\set{\eta=1}$ having zero measure under $\mu$. On the other hand, the left tail bound \eqref{i:mu.ell.left.tail.bound} does not go below $\exp\{-2^{k/4}\}$, which means that $\mu(\eta=0)$ can be positive. Although the precise bound \eqref{i:mu.ell.left.tail.bound} is likely suboptimal, we point out that, under the conditions of Lemma~\ref{l:mu.ell.tail.bounds}, $\mu^\ell(\eta=0)$ must indeed be positive for all $\ell\ge1$. This is simply because it holds with positive probability that $d^\minus$ is zero, which implies that the empty product $\Pi^\minus$ is one, which in turn implies that $R(\ud,\ueta)=0$. The probability for $d^\minus$ to be zero is
	\[
	\P\bigg(\Pois\bigg(\f{\alpha k}{2}\bigg)=0\bigg)
	= \f1{\exp(\alpha k/2)}
	\ge \f1{\exp( k2^k)}\,,
	\]
so $\mu^\ell(\eta=0)\ge \exp(-k2^k)$ for all $\ell\ge1$.
\end{rmk*}

\begin{lem}[bounds on clause recursion] \label{l:X.tail.bounds}
Let $\mu$ be any probability measure over $\eta\in[0,1)$, and let $\bhmu=\hREC_\EPS\mu$
for $\hREC_\ep$ as given by Definition~\ref{d:distributional.clause.recursion}. Sample $\bhu\sim\bhmu$, and note that $\bhu(\plus)\in[0,1)$ almost surely. Let
	\[
	X \equiv \log \f1{1-\bhu(\plus)}
	\ge \bhu(\plus)
	\ge0\,.
	\]
Write $\P$ and $\E$ for probability and expectation over the law of $X$. If $\mu$ satisfies the conditions \eqref{i:mu.ell.upper.tail.bound} and \eqref{i:mu.ell.left.tail.bound} of Lemma~\ref{l:mu.ell.tail.bounds}, then $X$ satisfies the following estimates:
\begin{enumerate}[(a)]
\item \label{i:X.left.tail} $\P( X \le (2^{-\xi}/2)^{k-1}) \le \exp\{-\Omega(\xi 2^{k/4})\}$ for all $k/2^{k/4} \le \xi\le1$.
\item \label{i:X.right.tail.above.one}
For all $x\ge1$, $\P(X\ge x) \le \exp\{-\Omega(kx2^{k/4})\}$;
\item \label{i:X.right.tail.one.tenth}
$\P(X\ge (2^{1/10}/2)^{k-1}) \le \exp\{ -\Omega(k2^{k/4})\})$;
\item \label{i:e.X.estimate}
$\E X = [1 + O(k/2^{k/8})]/2^{k-1}$;
\item \label{i:e.X.truncation.estimate}
$\E[X\Ind{X \ge (2^{1/10}/2)^{k-1}}]\le \exp\{-\Omega(k2^{k/4})\}$.
\end{enumerate}
The estimates of this lemma hold uniformly over all $0\le\EPS\le1$.

\begin{proof}
We can sample $X$ as follows: let $(\eta_j)_{j\ge1}$ be a sequence of i.i.d.\ samples from $\mu$, let $K\sim\clausedeg_\EPS$, and let
	\[X \equiv -\log\bigg(1 - \prod_{j=1}^{K-1}\eta_j\bigg)
	\ge \prod_{j=1}^{K-1}\eta_j
	\equiv \hat{u} 
	\ge \prod_{j=1}^{k-1}\eta_j
	\ge0\,.\]
If $0\le \xi\le1$ then $2^{-\xi}/2 \le 1/2-\xi/4$. Consequently, for all 
$4/2^{k/4} \le \xi \le 1$, we find
	\[
	\P\bigg( 
	\hat{u} \le \bigg(\f{2^{-\xi}}{2}\bigg)^{k-1} \bigg)
	\le
	\P\bigg(
	\min_{j\in[k-1]}
	\eta_j \le \f{2^{-\xi}}{2}
	\bigg)
	\le k \mu\bigg(\eta \le \f{2^{-\xi}}{2}\bigg)
	\le k\mu\bigg(
	\eta \le \f12-\f{\xi}{4}\bigg)
	\le \f{k}{\exp\{ \Omega(\xi 2^{k/4}) \}}\,,\]
where the last step is by the assumed lower bound~\eqref{i:mu.ell.left.tail.bound}. 
Since $X\ge \hat{u}$, we obtain,
for all $4/2^{k/4}\le\xi\le1$, that
	\beq\label{e:X.left.tail}
	\P\bigg(X \le \bigg(\f{2^{-\xi}}{2}\bigg)^{k-1} \bigg)
	\le
	\P\bigg( 
	\hat{u} \le \bigg(\f{2^{-\xi}}{2}\bigg)^{k-1} \bigg)
	\le \f{k}{\exp\{ \Omega(\xi 2^{k/4}) \}}\,.
	\eeq
This implies \eqref{i:X.left.tail}. We next consider the right tail of $X$. We have $x + \log(1-e^{-x}) \ge \Omega(x)$ uniformly over $x\ge1$, so
	\begin{align*}
	\P(X \ge x)
	&=\P\bigg(\hat{u}=\prod_{j=1}^{K-1}\eta_j 
	\ge 1-e^{-x}\bigg)
	\le \P\bigg( 
		\max_{1\le j\le k-2} \eta_j \ge 1-e^{-x}
		\bigg) \\
	&= \mu\bigg(
		\log\f{\eta}{1-\eta}
		\ge x + \log(1-e^{-x})
		\bigg)^{k-2}
		\le \f1{\exp\{\Omega( kx 2^{k/4} ) \}}
	\end{align*}
for all $x\ge1$, where the last step follows by using the assumed upper bound \eqref{i:mu.ell.upper.tail.bound}. This proves \eqref{i:X.right.tail.above.one}; and we obtain
	\beq\label{e:X.above.one.expectation}
	\E\bigg(X \Ind{X\ge1}\bigg)
	\le \f1{\exp\{\Omega( k 2^{k/4} ) \}}
	\eeq
by integrating \eqref{i:X.right.tail.above.one} over $x\ge1$. Next, for $0<\xi<1$, we have
	\[
	\P\bigg(\hat{u}\ge \bigg(\f{2^\xi}{2}\bigg)^{k-1}
	\bigg)
	\le\P\bigg(
	\eta_j \ge \f{2^{\xi/2}}{2}
	\textup{ for at least $\f{(k-2)\xi}2$ indices
	$j\in[k-2]$}
	\bigg)
	\le k^{k\xi}
	\mu\bigg(
	\eta\ge \f{2^{\xi/2}}{2}
	\bigg)^{(k-2)\xi/2}\,,\]
where the factor $k^{k\xi}$ upper bounds the choice of indices $j$. We then note that
	\[s(\xi)
	\equiv \log\f{2^\xi/2}{1-2^\xi/2}
	\ge (2\log2)\xi\]
by calculus. Substituting into the previous bound gives,
for all $2^{-k/4} \le \xi<1$,
	\beq\label{e:hat.u.upper.tail.small}
	\P\bigg(\hat{u}\ge \bigg(\f{2^\xi}{2}\bigg)^{k-1}
	\bigg)
	\le k^{k\xi}
	\mu\bigg(
	\log \f{\eta}{1-\eta} \ge (2\log2)\xi
	\bigg)^{(k-2)\xi/2}
	\le
	\f{k^{k\xi}}{\exp\{\Omega(k\xi^2 2^{k/4})\}}
	\,,
	\eeq
having again used the assumed upper bound \eqref{i:mu.ell.upper.tail.bound}. We next note that if $X \ge (2^{2\xi}/2)^{k-1}$, then
	\[
	\hat{u}=1-e^{-X}
	\ge 
	1-\exp\bigg\{-
		\bigg(\f{2^{2\xi}}{2}\bigg)^{k-1}
	\bigg\}
	= \bigg(\f{2^{2\xi}}{2}\bigg)^{k-1}
	\exp\bigg\{
	- O\bigg( \f{2^{2\xi}}{2}\bigg)^{k-1} \bigg)
	\bigg\}
	\ge \bigg(\f{2^\xi}{2}\bigg)^{k-1}\,,
	\]
where the last inequality holds if we restrict to $\xi\le1/4$. Combining with \eqref{e:hat.u.upper.tail.small} gives
	\beq\label{e:X.upper.tail.small}
	\P\bigg(X
	\ge \bigg(\f{2^{2\xi}}{2}\bigg)^{k-1}\bigg)
	\le
	\f{k^{k\xi}}{\exp\{\Omega(k\xi^2 2^{k/4})\}}
	\le
	\f1{\exp\{\Omega(k\xi^2 2^{k/4})\}}
	\eeq
for all $k/2^{k/4} \le \xi \le 1/4$, of which a special case is the claimed bound \eqref{i:X.right.tail.one.tenth}. Combining \eqref{e:X.left.tail}, \eqref{e:X.above.one.expectation}, and \eqref{e:X.upper.tail.small} gives \eqref{i:e.X.estimate}. Finally, we can combine \eqref{i:X.right.tail.one.tenth} with \eqref{e:X.above.one.expectation} to bound
	\[0\le \E\bigg[X\mathbf{1}\bigg\{X \ge \bigg(
			\f{2^{1/10}}{2}
		\bigg)^{k-1} \bigg\}\bigg]
	\le \P\bigg( X \ge 
		\bigg(
			\f{2^{1/10}}{2}
		\bigg)^{k-1}\bigg)
		+\E\bigg( X \Ind{X\ge1} \bigg)\\
	\le \f1{\exp\{ \Omega(k2^{k/4}) \}}\,,\]
which gives \eqref{i:e.X.truncation.estimate}.
\end{proof}
\end{lem}

\begin{proof}[Proof of Lemma~\ref{l:mu.ell.tail.bounds}]
First we recall that if $D$ is a Poisson random variable with mean $\lambda$, then a standard Chernoff bound gives, with $f(u)\equiv (1+u)\log(1+u)-u$,
	\beq\label{e:poisson.chernoff.bound}
	\P(|D-\lambda|\ge y)
	\le 2\exp\bigg\{
	-\lambda
	\min\bigg\{
	f\bigg(\f{y}{\lambda}\bigg)\,,
	f\bigg(-\f{y}{\lambda}\bigg)
	\bigg\}
	\bigg\}\,.\eeq
Suppose inductively that the bounds
\eqref{i:mu.ell.upper.tail.bound}~and~\eqref{i:mu.ell.left.tail.bound} hold for $\mu^\ell$.
Sample $\ud=(d^\plus,d^\minus)$ from $\POpm$
and $\ueta$ from $(\mu^\ell)^{\otimes\infty}$, 
so that $\eta^{\ell+1}=R(\ud,\ueta)$ has law $\mu^{\ell+1}$. We then expand
	\beq\label{e:log.eta.ratio.in.terms.of.Pi}
	\log \f{\eta^{\ell+1}}{1-\eta^{\ell+1}}
	=\log \f{ \Pi^\plus(1-\Pi^\minus) }{\Pi^\minus}
	=\log(1-\Pi^\minus)
	+\log\f1{\Pi^\minus}
	-\log\f1{\Pi^\plus}\,.\eeq
Conditional on $\ud$ we can decompose $\Sigma^\PM\equiv\log(1/\Pi^\PM)$ as a sum of i.i.d.\ terms:
	\beq\label{e:Pi.as.sum.of.X}
	\Sigma^\PM
	\equiv\log \f1{\Pi^\PM}
	= \sum_{i=1}^{d^\PM}
	\bigg\{-\log \bigg(
			1 - \prod_{j=1}^{k-1} \eta^\PM_{ij}
			\bigg)\bigg\}
	\equiv \sum_{i=1}^{d^\PM} X_i^\PM
	\ge0\,,\eeq
where the $X_i^\PM$ are equidistributed as the random variable $X$ of Lemma~\ref{l:X.tail.bounds}. The remainder of the proof is divided into a few steps.\smallskip

\noindent\bemph{Step 1. Concentration bounds for $\Sigma^\PM$.} Consider $\Sigma^\plus$, and define the truncated random variables
	\[Y_i\equiv X^\plus_i
	\mathbf{1}\bigg\{
	 X^\plus_i \le \bigg(\f{2^{1/10}}{2}\bigg)^{k-1}
	 \bigg\}\,,\quad
	Y\equiv X
	\mathbf{1}\bigg\{
	X \le \bigg(\f{2^{1/10}}{2}\bigg)^{k-1}
	 \bigg\}\,.\]
Next define the events
	\[ \BIG
	\equiv\bigg\{
	\max_{i\le \alpha k}
	X^\plus_i \ge \bigg(\f{2^{1/10}}{2}\bigg)^{k-1}
	\bigg\}\,,\quad
	\DEG
	\equiv 
	\bigg\{\bigg| d^\plus-\f{\alpha k}{2}\bigg|
	\ge k^2 2^{5k/8}\bigg\}\,.
	\]
On the complement of the event $\DEG\cup\BIG$, 
it follows using Lemma~\ref{l:X.tail.bounds}\eqref{i:e.X.truncation.estimate} that
	\begin{align*}
	\bigg|\sum_{i=1}^{d^\plus} X^\plus_i
	- \f{\alpha k}{2}\E X
	\bigg|
	&=\bigg|\sum_{i=1}^{d^\plus} Y_i
	- \f{\alpha k}{2}\E X
	\bigg|
	\le
	\bigg|
	\sum_{i=1}^{\alpha k/2} Y_i
	- \f{\alpha k}{2}\E Y
	\bigg|
	+\f{k^2 2^{5k/8}}{2^{9k/10}}
	+ \alpha k \E(X-Y)
	\\
	&\le\bigg|
	\sum_{i=1}^{\alpha k/2} Y_i
	- \f{\alpha k}{2}\E Y
	\bigg|
	+\f{k^2 2^{5k/8}}{2^{9k/10}} + \f{\alpha k}{\exp\{\Omega(k2^{k/4})\}}
	\le
	\bigg|
	\sum_{i=1}^{\alpha k/2} Y_i
	- \f{\alpha k}{2}\E Y
	\bigg|
	+ \f{1}{k^5 2^{k/4}}\,.
	\end{align*}
Consequently, for $t \ge 1/(k2^{k/4})$, we obtain by a union bound that
	\[
	\P\bigg(
	\bigg|
	\sum_{i=1}^{d^\plus} X^\plus_i
	- \f{\alpha k}{2}\E X
	\bigg|\ge t\bigg)
	\le\P(\BIG\cup\DEG)+
	\P\bigg(
	\bigg|
	\sum_{i=1}^{\alpha k/2} Y_i
	- \f{\alpha k}{2}\E Y
	\bigg|\ge \f{t}{2}\bigg)\,.
	\]
By Lemma~\ref{l:X.tail.bounds}\eqref{i:X.right.tail.one.tenth}
and a trivial union bound over $1\le i\le \alpha k$, we have
	\[\P(\BIG) \le 
	\f{\alpha k }{\exp\{\Omega(k2^{k/4})\}}
	\le \f1{\exp\{\Omega(k2^{k/4})\}}\,.\]
In the Poisson Chernoff bound
\eqref{e:poisson.chernoff.bound},
for small $u$ we have $f(u)\asymp u^2$, so
$\P(\DEG)\le\exp\{-\Omega(k^3 2^{k/4})\}$. Finally, by the Azuma--Hoeffding inequality, it holds for all $t\ge0$ that
	\beq\label{e:plain.azuma.hoeffding.bound}
	\P\bigg(
	\bigg|\sum_{i=1}^{\alpha k/2} (Y_i-\E Y_i)\bigg|
	\ge t\bigg)
	\le \f{O(1)}{\exp\{\Omega(t^2 2^{4k/5}/k)\}}\,.
	\eeq
Recall from $\eqref{e:Pi.as.sum.of.X}$ that the sum of $X^\plus_i$ over $1\le i\le d^\plus$ is exactly $\Sigma^\plus$. Combining the above bounds gives 
	\beq\label{e:azuma.hoeffding.on.trunc}
	\P\bigg(
	\bigg|\Sigma^\plus
	- \f{\alpha k}{2}\E X
	\bigg|\ge t\bigg)
	\le \f1{\exp\{ \Omega(k 2^{k/4}) \}}
	+\f1{\exp\{ \Omega( t^2 2^{4k/5}/k) \}}
	\,.\eeq
The same bound holds for $\Sigma^\minus$, which is equidistributed as $\Sigma^\plus$.

\smallskip\noindent\bemph{Step 2. Proof of right tail bound in moderate deviations regime.} We now rewrite the decomposition~\eqref{e:log.eta.ratio.in.terms.of.Pi} as
	\beq\label{e:log.eta.ratio.rewrite}
	\log \f{\eta^{\ell+1}}{1-\eta^{\ell+1}}
	= \log\bigg(1-\f1{\exp \Sigma^\minus}\bigg)
	+\Sigma^\minus-\Sigma^\plus
	\le |\Sigma^\minus-\Sigma^\plus|
	\le 
	\bigg|\Sigma^\plus-\f{\alpha k}{2}\E X\bigg|
	+\bigg|\Sigma^\minus-\f{\alpha k}{2}\E X\bigg|
	\,.
	\eeq
We can then apply the preceding bound~\eqref{e:azuma.hoeffding.on.trunc} to obtain
that
for all $k/2^{11k/40} \le s \le 1$,
	\[
	\mu^{\ell+1}\bigg(
	\log \f{\eta^{\ell+1}}{1-\eta^{\ell+1}} \ge s 
	\bigg)
	\le \P\bigg(
		|\Sigma^\minus-\Sigma^\plus| \ge s
	\bigg)
	\le 2\P\bigg(
	\bigg|\Sigma^\minus-\f{\alpha k}{2}\E X\bigg|
	\ge \f{s}{2}
	\bigg)
	\le \f1{\exp\{ \Omega( s k 2^{k/4}) \}}\,.
	\]
This implies the desired upper bound \eqref{i:mu.ell.upper.tail.bound} for $2^{-k/4} \le s \le 1$.

\smallskip\noindent\bemph{Step 3. Proof of left tail bound.} 
We will show $\mu(\eta\le1/2-s)\le\exp(-\Omega(sk2^{k/4}))$ for all $2^{-k/4}\le s\le 1/2$; the bound \eqref{i:mu.ell.left.tail.bound} then follows. For $0\le s\le 1/2$, note that $\eta\le1/2-s$ if and only if
	\[
	\log\f{\eta}{1-\eta}
	\le \log\f{1-2s}{1+2s} \le -4s\,.
	\]
Recall the 
decomposition~\eqref{e:log.eta.ratio.rewrite}
for $\log[\eta^{\ell+1}/(1-\eta^{\ell+1})]$.
It follows that
	\beq\label{e:proof.of.left.tail.bound}
	\mu^{\ell+1}\bigg(
	\eta^{\ell+1}\le\f12-s
	\bigg)
	\le\P\bigg(
	\log\bigg(1-\f1{\exp \Sigma^\minus}\bigg)
	\le -2s
	\bigg)+\P\bigg(\Sigma^\minus-\Sigma^\plus \le -2s
	\bigg)\,.\eeq
On the right-hand side of
\eqref{e:proof.of.left.tail.bound},
we bound the first term by noting that
	\[\P\bigg(
	\log\bigg(1-\f1{\exp \Sigma^\minus}\bigg)
	\le -2s
	\bigg)
	=
	\P\bigg(
	\Sigma^\minus\le\log \f{1}{1-\exp(-2s)}
	\le \log \f1s
	\bigg)\,,
	\]
again for all $0\le s\le 1/2$.
Recall Lemma~\ref{l:X.tail.bounds}\eqref{i:e.X.estimate}, and let $\bar{s}$ be the solution to the equation
	\[
	\log\f1{\bar{s}}
	= \f{\alpha k}{2} \E X
	= \f{2^k k\log2}{2}
	\f{1 + O(k/2^{k/8})}{2^{k-1}}
	\ge \f{k\log2}{2}\,,
	\]
so $\bar{s} \le 2^{-k/2}$. Then for $2^{-k/4}\le s\le 1/2$ we have $s'\equiv s-\bar{s} \ge s/2$, It follows that 
the first term on the right-hand side of \eqref{e:proof.of.left.tail.bound}
is upper bounded by
	\beq\label{e:left.tail.Sigma}
	\P\bigg(\Sigma^\minus\le\log\f1s\bigg)
	=\P\bigg(\Sigma^\minus- \f{\alpha k}{2} \E X
	\le\log\f1{\bar{s}+s/2}-\log\f1{\bar{s}}
	\le -s\bigg)
	\le \f1{\exp\{ \Omega(sk2^{k/4}) \}}\,,
	\eeq
where the last step uses the earlier estimate on $\Sigma^\minus$. The same estimate implies that the second term on the right-hand side of \eqref{e:proof.of.left.tail.bound} is also at most $\exp\{-\Omega(sk2^{k/4})\}$; and this concludes the verification of the lower bound~\eqref{i:mu.ell.left.tail.bound}.

\smallskip\noindent\bemph{Step 4. Proof of right tail bound in large deviations regime.} We now verify \eqref{i:mu.ell.upper.tail.bound} for $s\ge1$. Define the event
	\[
	\DEG^s
	\equiv \bigg\{ \bigg|d^\plus-\f{\alpha k}{2}\bigg|
	\ge s2^{3k/4} \bigg\}\,.
	\]
Now note that in the Poisson Chernoff bound \eqref{e:poisson.chernoff.bound}, the function $f(u)=(1+u)\log(1+u)-u\ge0$ is strictly convex with respect to $u\in(-1,\infty)$. Consequently, if $u\ge 2 u_0 > 0$, then $f(u) \ge f(u_0) + f'(u_0) (u-u_0) \ge f'(u_0) u/2$. Likewise, if $u \le 2u_0 < 0$, then $f(u) \ge f'(u_0) u/2$. 
We note also that $f'(u)=\log(1+u)$, so if $u_0$ is small then $f'(u_0)\asymp u_0$. We can therefore conclude that for all $s\ge1$, we have
	\[
	\P(\DEG^s)
	\le \f1{\exp\{\Omega(s2^{k/2}/k)\}}\,.
	\]
Recall from above that $Y_i$ denotes the truncated version of $X^\plus_i$. On the event $\DEG^s$ we have
	\begin{align*}
	&\bigg|\sum_{i=1}^{d^\plus} Y_i
	-\f{\alpha k}{2}\E X\bigg|
	\le \bigg|\sum_{i=1}^{\alpha k/2}
		(Y_i-\E Y_i)\bigg|
	+\bigg|\sum_{i=1}^{d^\plus} Y_i
	- \sum_{i=1}^{\alpha k/2} Y_i \bigg|
	+ \f{\alpha k}{2}\E(X-Y)\\
	&\qquad\le \bigg|\sum_{i=1}^{\alpha k/2}
		(Y_i-\E Y_i)\bigg|
		+ s 2^{k/4} \bigg( \f{2^{1/10}}2 \bigg)^{k-1}
		+ \f1{\exp\{\Omega(k2^{k/4})\}}
	\le \bigg|\sum_{i=1}^{\alpha k/2}
		(Y_i-\E Y_i)\bigg|
		+ \f{s}{2^{k/2}}
	\,,
	\end{align*}
where we used Lemma~\ref{l:X.tail.bounds}\eqref{i:e.X.truncation.estimate} to bound $\E(X-Y)$,
and the last bound holds for all $s\ge1$.
Then, by combining the bound on $\P(\DEG^s)$ with the Azuma--Hoeffding bound~\eqref{e:plain.azuma.hoeffding.bound}, we find
	\beq\label{e:azuma.hoeffding.with.large.dev}
	\P\bigg(
	\bigg|\sum_{i=1}^{d^\plus} Y_i
	- \f{\alpha k}{2} \E X \bigg|
	\ge \f{s}{k}
	\bigg)
	\le \f1{\exp\{\Omega(s2^{k/2}/k)\}}
	+ \f1{\exp\{ \Omega( s^2 2^{4k/5} /k^3 )\}}
	\le \f1{\exp\{\Omega(s2^{k/2}/k)\}}\,.
	\eeq
Next we account for the difference between $Y_i$ and $X^\plus_i$. Recall Lemma~\ref{l:X.tail.bounds} parts~\eqref{i:X.right.tail.above.one}~and~\eqref{i:X.right.tail.one.tenth}. Together they imply,
 with $\preccurlyeq$ denoting stochastic domination, that
	\beq\label{e:stoch.dom}
	\sum_{i=1}^{d^\plus}
	(X^\plus_i-Y_i)
	\preccurlyeq
	\sum_{i=1}^{d^\plus}I_i(1+Z_i)\eeq
where $I_i$ are i.i.d.\ $\textup{Ber}(e^{-\theta})$ indicators, and the $Z_i$ are i.i.d.\ $\theta^{-1}\textup{Exp}$ random variables
(where $\textup{Exp}$ denotes a standard exponential random variable), with $\theta=ck2^{k/4}$ for some absolute constant $c>0$. By Poisson thinning,
	\[
	A\equiv
	\sum_{i=1}^{d^\plus} I_i
	\sim \Pois\bigg(
	\f{\alpha k}{2e^\theta}
	\equiv\lambda'\bigg)\,.
	\]
Note that $\lambda'$ is very small, and it follows from 
\eqref{e:poisson.chernoff.bound} that for all $s\ge1$,
	\[
	\P\bigg(A \ge \f{s}{k^{1/2}}\bigg)
	\le
	\exp\bigg\{
	-\Omega\bigg(
	\f{s}{k^{1/2}}\log \f{s/k^{1/2}}{\lambda'} \bigg)
	\bigg\}
	\le\f1{\exp\{\Omega(s\theta/k^{1/2})\}}
	\le \f1{\Omega(s k^{1/2} 2^{k/4})}\,.
	\]
Conditioned on $A$, the other term of \eqref{e:stoch.dom}
is distributed as a gamma random variable with shape parameter $A$,
	\[
	B\equiv
	\sum_{i=1}^{d^\plus} I_i Z_i
	\sim \f{\textup{Gamma}(A)}{\theta}\,.\]
Let $E$ be a standard exponential random variable,
independent of $A$. The moment-generating function of $B$ is
	\[
	m(t) = \E(e^{tB})
	= \E\bigg[\bigg\{\E \exp\bigg(\f{t E}{\theta}\bigg)\bigg\}^A\bigg]
	= \E \bigg[\bigg\{\f1{1-t/\theta}\bigg\}^A\bigg]
	= \exp\bigg\{
	\f{\lambda't}{\theta-t}
	\bigg\}\,,
	\]
for $t<\theta$. For $x>\E B=\lambda'/\theta$, optimizing over $t$ gives
	\[
	\P(B\ge x)
	\le
	\exp\bigg\{ - \theta x
	\bigg[ 1- \bigg(
	\f{\lm'}{\theta x}\bigg)^{1/2} \bigg]^2 \bigg
	\}\,,
	\]
which implies $\P(B \ge s/k^{1/2}) \le \exp\{-\Omega(sk^{1/2} 2^{k/4})\}$ for all $s\ge1$. Combining these estimates on $A$ and $B$ with our earlier bound~\eqref{e:azuma.hoeffding.with.large.dev} gives altogether
	\[
	\P\bigg(\bigg|
	\Sigma^\plus
	-\f{\alpha k}{2}\E X\bigg|
	\ge \f{s}{2}
	\bigg)
	\le \f1{\exp\{\Omega(sk^{1/2} 2^{k/4})\}}\,.
	\]
The same bound holds for $\Sigma^\minus$.
Substituting into \eqref{e:log.eta.ratio.rewrite},
we obtain the desired bound
\eqref{i:mu.ell.upper.tail.bound} for all $s\ge1$.

\smallskip\noindent\bemph{Step 5. Conclusion.}
Having verified \eqref{i:mu.ell.upper.tail.bound} and \eqref{i:mu.ell.left.tail.bound}, it remains to show
\eqref{i:mu.ell.eta.of.free.estimate}. Recall that we can express
	\[
	\bmeta_{\ell+1}(\free)
	=\f{\Pi^\plus\Pi^\minus}
		{\Pi^\plus+\Pi^\minus-\Pi^\plus\Pi^\minus }
	= \f1{
	\exp(\Sigma^\plus)+\exp(\Sigma^\minus)
	-1}\,.
	\]
It follows from \eqref{e:azuma.hoeffding.on.trunc}
that $\Sigma^\plus$ is concentrated near $(\alpha k/2) \E X$, which by Lemma~\ref{l:X.tail.bounds}\eqref{i:e.X.estimate} is close to $k\log2$:
	\beq\label{e:simple.bound.on.Sigma}
	\P\bigg(\Big|\Sigma^\plus- k\log2\Big|
		\ge \f{k}{2^{11k/40}}
		\bigg)
	\le \f1{\exp\{\Omega(k 2^{k/4})\}}\,,
	\eeq
and likewise for $\Sigma^\minus$. It follows that
	\[
	\bmu^{\ell+1}\bigg(
	|2^{k+1}\bmeta_{\ell+1}(\free)-1|
	\ge \f1{2^{k/4}}
	\bigg)
	\le
	2\P\bigg(\Big|\Sigma^\plus- k\log2\Big|
		\ge \f{k}{2^{11k/40}}
		\bigg)
	\le \f1{\exp(2^{k/4})}\,,
	\]
which concludes the proof of \eqref{i:mu.ell.eta.of.free.estimate}.
\end{proof}

\subsection{Niceness in the Galton--Watson tree} 
\label{ss:root.is.robust}

The main goal of this subsection is to prove that under the $\uPGW$ measure, the root variable fails to be $1$-nice with very small probability. For technical reasons (which will emerge in the proof of Lemma~\ref{lem-local-D-0}), we will prove a version of this statement which is slightly stronger in two ways. First, we generalize from $\uPGW$ to $\uPGW(T)$ (Definition~\ref{d:PGW.based.on.T}) where $T$ is very sparse, with maximum degree $O(1)$ (not growing with $k$) --- since typical degrees in $\uPGW$ diverge with $k$, it is intuitively plausible that planting the sparse subtree $T$ cannot have a large effect. Second, we replace $1$-nice with a more restrictive property which we now define:

\begin{dfn}[robustness]\label{d:J.robust}
If $U$ is any rooted tree and $x$ is any vertex in $U$, we let $U(x)$ denote the subtree of $U$ that lies below $x$. Given any tree $\tree$ rooted at a variable $\vrt$, we say that $\ptree$ is a \bemph{$\CC$-modification} of $\tree$ if $\ptree$ can be obtained from $\tree$ by
\begin{enumerate}[--]
\item deleting at most one subtree $\tree(u)$ for $u\in N(\vrt)$, and 
\item changing at most $\CC$ subtrees $\tree(u)$ for $u\in\pd_2\vrt$.
\end{enumerate}
The new subtrees $\ptree(u)$, for $u\in\pd_2\vrt$, can be arbitrary. We then say that an acyclic variable $v$ is \bemph{$\CC$-robust} if every $\CC$-modification of $B_r(v)$ is nice. Note that being $\CC$-robust is stronger than being $1$-nice.
\end{dfn}

\begin{ppn}\label{p:uPGW.is.J.robust}
Let $\CC$ be an absolute constant. Let $\P=\uPGW(T)$ where $T$ is a fixed bipartite factor tree, rooted at a variable, with maximum vertex degree at most $\CC$. For $\tree\sim\P$,
	\[
	\P\Big(\textup{$\tree$ is not $\CC$-robust}\Big)
	\le \f1{\exp(\Omega(k2^{k/4}))}\,.
	\]
(We allow $T=\emptyset$, in which case the statement is for $\P=\uPGW(\emptyset)=\uPGW$.) 
\end{ppn}

In fact, this result is a relatively straightforward consequence of the technical lemmas of the previous \S\ref{ss:technical.bounds.on.dist.rec}. 
We first consider the effect of changing a small number of subtrees. If $\tree_{\vrt\crt}$ and $\ptree_{\vrt\crt}$ are both variable-to-clause trees, we will say that they are \bemph{$\CC$-perturbations} of one another if $\ptree_{\vrt\crt}$ can be obtained from $\tree_{\vrt\crt}$ by only changing subtrees $\tree_{\vrt\crt}(u)$ for at most $\CC$ variables $u\in N(\vrt)$. The new subtrees $\ptree_{\vrt\crt}(u)$ are allowed to be arbitrary.\footnote{Note that the $\CC$-perturbation defined here is different from the $\CC$-modification of Definition~\ref{d:J.robust}.} The next lemma says, essentially, that $\CC$-perturbations have very little effect on the outgoing variable-to-clause message. The formal statement is as follows:

\begin{lem}\label{l:perturb.below.variables} Let $\P=\PGW(T)$, where $T\equiv T_{\vrt\crt}$ is any fixed variable-to-clause tree of maximum vertex degree at most $\CC$. We allow $T=\emptyset$, in which case $\P=\PGW(\emptyset)=\PGW$. If $\tree\equiv\tree_{\vrt\crt}$ is sampled from $\P$, then
	\[\P\Bigg(
	\max\Bigg\{
	\sum_{x\in\set{\plus,\minus,\free}}
	\Big| 2^{k\Ind{x=\free}}
	 \bmeta'(x)-\f12\Big|
		\,\Bigg|\,\hspace{-3pt}
	\begin{array}{c}
	\textup{$\bmeta' =\FF_\ell(\ptree_{\vrt\crt})$,
	for $\ptree_{\vrt\crt}$}
		\\
	\textup{a $\CC$-perturbation
		of $\tree_{\vrt\crt}$}
	\end{array}
	\hspace{-3pt}
	\Bigg\}	\ge \f1{2^{k/4}}
	\Bigg)
	\le \f1{\exp(\Omega(k2^{k/4}))}
	\]
for all $\ell\ge0$.

\begin{proof} Throughout this proof we will abbreviate $T\equiv T_{\vrt\crt}$, $\tree\equiv\tree_{\vrt\crt}$, and $\ptree\equiv\ptree_{\vrt\crt}$. From the definitions (see, in particular, the discussion below \eqref{e:def.FF.of.tree}), the measure $\bmeta\equiv\FF_\ell(\tree)$ only depends on the tree up to depth $\ell$ below $\vrt$. For $0\le \ell\le1$, any $\CC$-perturbation has $\bmeta'=\bmeta$, so there is nothing more to prove. We therefore assume $\ell\ge2$ for the rest of the proof.

We next describe a procedure to generate a sample of $\bmeta$. In the fixed tree $T$, we partition the first layer of clauses (at depth $1/2$ below $\vrt$) according to the signs on the edges from these clauses to $\vrt$:
	\[
	(\pd^\plus T,\pd^\minus T)
	\equiv
	\bigg( T \cap \pd\vrt(\plus\crt)\,,
	T \cap \pd\vrt(\minus\crt)\bigg)\,.
	\]
Let $d^\PM(T)\equiv|\pd^\PM T|$; these are both upper bounded by $\CC$. (If $T=\emptyset$ we define $d^\PM(T)\equiv0$.) For $1\le i\le d^\plus(T)$, the $i$-th clause in $\pd^\plus T$ has $\CC^\plus(i)$ child variables in $T$. Let $T(\plus,i,j)$ be the subtree of $T$ descended from the $j$-th child variable of the $i$-th clause in $\pd^\plus(T)$. Define similarly $\CC^\minus(i)$ and $T(\minus,i,j)$. Now define $\bmeta$ by the following steps:
\begin{enumerate}[(i)]
\item Sample an array of independent random trees
$\tree^\PM_{ij}\sim\PGW(T(\PM,i,j))$,
and define the corresponding messages
	\[
	\eta^\PM_{ij}(T)\equiv
	\Big[
	\FF_{\ell-1}(\tree^\PM_{ij})
	\Big](\minus)\,,
	\]
for $1\le i\le d^\PM(T)$ and $1\le j\le \CC^\PM(i)$.

\item Let $\ueta$ be an array (as in \eqref{e:defn.array.ueta}) of i.i.d.\ samples from $\mu^{\ell-1}=
\Rec^{\ell-1}\mu^0$, where $\mu^0=\delta_{1/2}$ as before. Let
	\[
	u^\plus_i(T)
	\equiv \prod_{j=1}^{\CC^\plus(i)} 
	\eta^\plus_{ij}(T)
	\prod_{j=\CC^\plus(i)+1}^{k-1}
	\eta^\plus_{ij}\,,\quad
	u^\plus_i
	\equiv \prod_{j=1}^{k-1} \eta^\plus_{ij}\,.
	\]
Define similarly $u^\minus_i(T)$ and $u^\minus_i$.

\item Let $\ud\equiv(d^\plus,d^\minus)$ be an independent sample from $\POpm$, and define
	\beq\label{e:Pi.with.embedded.T}
	\Pi^\plus(T)
	\equiv\prod_{i=1}^{d^\plus(T)}
	\bigg(1-u^\plus_i(T)\bigg)\,,\quad
	\Pi^\plus
	\equiv\prod_{i=d^\plus(T)+1}^{d^\plus(T)+d^\plus}
	\bigg(1-u^\plus_i\bigg)\,,\quad
	\Pi^{*\plus}
	\equiv
	\Pi^\plus(T)\Pi^\plus\,.
	\eeq
Define similarly $\Pi^\minus(T)$, $\Pi^\minus$, and $\Pi^{*\minus}=\Pi^\minus(T)\Pi^\minus$.
\item Substitute the $\Pi^{*\PM}$ into \eqref{e:eta.in.terms.of.signed.Pis} to define $\bmeta$:
	\[
	\bigg(
	\bmeta(\plus),\bmeta(\minus),\bmeta(\free)
	\bigg)
	=
	\bigg(\f{\Pi^{*\minus}(1-\Pi^{*\plus})}
		{ \Pi^{*\plus} + \Pi^{*\minus} 
			- \Pi^{*\plus} \Pi^{*\minus} },
	\f{\Pi^{*\plus}(1-\Pi^{*\minus})}
		{ \Pi^{*\plus} + \Pi^{*\minus} 
			- \Pi^{*\plus} \Pi^{*\minus} },
	\f{\Pi^{*\plus}\Pi^{*\minus}}
		{ \Pi^{*\plus} + \Pi^{*\minus} 
			- \Pi^{*\plus} \Pi^{*\minus} }
	\bigg)\,.
	\]
\end{enumerate}
The $\bmeta$ that results from this construction can be regarded as a sample of $\FF_\ell(\tree)$ for $\tree\sim\PGW(T)$, although we did not explicitly generate all of $\tree$. Moreover, if $\bmeta'=\FF_\ell(\ptree)$ where $\ptree$ is any $\CC$-perturbation of $\tree$, then $\bmeta'$ can also be obtained from the above procedure by modifying at most $\CC$ of the messages from depth one.

For convenience we will let $\vec{\theta}$ denote the messages that are actually used in the definition of $\bmeta$, so
	\[
	\theta^\plus_{ij}
	\equiv\left\{
	\begin{array}{cl}
	\eta^\plus_{ij}(T)
	&\textup{if }1\le i\le d^\plus(T)
	\textup{ and }1\le j\le \CC^\plus(i);\\
		\eta^\plus_{ij}
	&\textup{otherwise,}
	\end{array}
	\right.
	\]
and similarly $\theta^\minus_{ij}$. To bound the different between $\bmeta$ and $\bmeta'$,
we first assume a \emph{fixed} set of affected indices: without loss of generality, let $\vec{\vartheta}$ denote a new set of messages such that
	\[
	\bigg\{
	(i,j):\vartheta^\PM_{ij}\ne\theta^\PM_{ij}
	\bigg\}
	\subseteq
	\bigg\{
	(i,j):
	1\le i\le 2\CC,
	1\le j\le 2\CC
	\bigg\}\,.
	\]
Returning to \eqref{e:Pi.with.embedded.T},
let $\Sigma^\PM(T)\equiv-\log \Pi^\PM(T)$,
$\Sigma^\PM\equiv-\log \Pi^\PM$,
and $\Sigma^{*\PM}\equiv -\log\Pi^{*\PM}$.
Let $\Xi^\PM(T)$, $\Xi^\PM$, and $\Xi^{*\PM}$
be defined analogously as the $\Sigma$ quantities,
but with $\vec{\vartheta}$ in place of $\vec{\theta}$. 
It is then straightforward to check that
	\begin{align*}
	|\Xi^{*\plus}
		-\Sigma^\plus|
	&\le
	\sum_{i=1}^{d^\plus(T)}
	-\log\bigg( 1-\prod_{j=1}^{k-1} \vartheta^\plus_{ij}\bigg)
	+\sum_{i=d^\plus(T)+1}^{2\CC}
	\bigg| \log \bigg( 1-\prod_{j=1}^{k-1} \vartheta^\plus_{ij}\bigg)
		-\log\bigg( 1-\prod_{j=1}^{k-1} \theta^\plus_{ij}
		\bigg)\bigg|\\
	&\le\sum_{i=1}^{2\CC}\bigg[\underbrace{
	-\log\bigg( 1-\prod_{j=2\CC+1}^{k-1} \theta^\plus_{ij} \bigg)}_{\Delta^\plus(i)}\bigg]
	\equiv \sum_{i=1}^{2\CC} \Delta^\plus(i)\,.
	\end{align*}
By the same argument as for the estimate~\eqref{e:X.upper.tail.small} from the proof of Lemma~\ref{l:X.tail.bounds}, we have
	\[
	\P\bigg(
	|\Xi^{*\plus}-\Sigma^\plus|
	\ge\bigg(\f{2^{1/20}}{2}\bigg)^k
	\bigg)
	\le 2\CC\cdot \P\bigg(
	|\Delta^\plus(i)|
	\ge\bigg(\f{2^{1/21}}{2}\bigg)^k
	\bigg)
	\le \f1{\exp\{\Omega(k2^{k/4})\}}\,,
	\]
and the same bound applies for $|\Xi^{*\minus}-\Sigma^\minus|$. Now recall the bound~\eqref{e:simple.bound.on.Sigma} from the proof of Lemma~\ref{l:mu.ell.tail.bounds}, which says that the $\Sigma^\PM$ are well concentrated around $k\log2$. 
In particular, it holds with probability at least 
$1-\exp(-\Omega(k2^{k/4}))$
that $|\Sigma^\PM-k\log2| \le k/2^{11k/40}$
and $|\Xi^{*\PM}-\Sigma^\PM| \le (2^{1/20}/2)^k$,
in which case
	\begin{align*}
	\bmeta'(\PM)
	&= \f{\exp(\Xi^{*\PM})-1}
		{\exp(\Xi^{*\PM})+\exp(\Xi^{*\PM})-1}
	= \f{1 + O(k/2^{11k/40})}2\,,\\
	\bmeta'(\free)
	&=\f1
		{\exp(\Xi^{*\PM})+\exp(\Xi^{*\PM})-1}
	=\f{1 + O(k/2^{11k/40})}{2^{k+1}}\,.
	\end{align*}
Thus, given a \emph{fixed} choice of at most $\CC$ perturbed indices, it holds with probability at least $1-\exp(-\Omega(k2^{k/4}))$ that any resulting message $\bmeta'$ satisfies the above estimates, regardless of how those $\CC$ indices are perturbed. To conclude, we note that by \eqref{e:poisson.chernoff.bound} the event $\DEG^+\equiv\{ \{ d^\plus,d^\minus \} \ge 4^k/k \}$ has probability upper bounded by $\exp(-\Omega(4^k))$. On the complement of $\DEG^+$, the number of distinct choices for the $\CC$ perturbed indices is $\exp(O(k))$,
so we can take a union bound over all choices to obtain the result.\end{proof}\end{lem}

\begin{proof}[Proof of Proposition~\ref{p:uPGW.is.J.robust}] Let $\tree$ be a sample of $\uPGW(T)$, rooted at variable $v\equiv\vrt$. Write $\pd v(\plus)$ and $\pd v(\minus)$ for the clauses neighboring the root $v$ in $\tree$, and let
	\[
	\DEG'
	\equiv
	\max\bigg\{
	\bigg| |\pd v(\plus)| - \f{2^k k\log2}{2}
	\bigg|,
	\bigg| |\pd v(\minus)| - \f{2^k k\log2}{2}
	\bigg|
	\bigg\} \ge k2^{5k/8}
	\]
It follows from \eqref{e:poisson.chernoff.bound}
that $\P(\DEG') \le \exp(-\Omega(k2^{k/4}))$.

Now, recalling Definition~\ref{d:J.robust}, we want to show that (with very good probability) any $\CC$-modification $\ptree$ of $\tree$ is nice in the sense of Definition~\ref{d:nice}. On the complement of the event $\DEG'$, it is clear that every $\CC$-modification $\ptree$ of $\tree$ will satisfy the degree condition \eqref{e:nice.degree.condition}. It remains to determine whether the canonical messages $\dqstar(\ptree)$ and $\hqstar(\ptree)$ satisfy the bounds \eqref{e:nice.var.message.cond} and \eqref{e:nice.clause.message.cond}.

To this end, let us fix $\ptree$ momentarily,
and abbreviate $q\equiv\dqstar(\ptree)$ and $\hq\equiv\hqstar(\ptree)$. Note that $q$ and $\hq$ are based the $r$-neighborhood $\ptree_r$ of the root $v$ in $\ptree$; see Definition~\ref{d:canonical}. For any edge $(au)$ in $\ptree$, let us abbreviate
	{\setlength{\jot}{0pt}\begin{align*}
	\bmeta_{ua}
	\equiv \bmeta(\ptree)_{ua}
	&\equiv \FF_{\ptree_r,ua}\,,\\
	\bhu_{au} \equiv\bhu(\ptree)_{au}
	& \equiv\FF_{\ptree_r,au}\,.
	\end{align*}}%
Recall from \eqref{e:color.recursions.eta} the
correspondence between $(\dq,\hq)$ and $(\bmeta,\bhu)$
for edges $(av)$ incident to the root $v$:
	\begin{align}
	\label{e:vtoc.q.given.eta}
	\Big(q_{va}(\red),q_{va}(\yel),
		q_{va}(\grn),q_{va}(\blu)\Big)
	&=\bigg(
	\f{\bmeta_{va}(\plus)+\bmeta_{va}(\free)}
		{2-\bmeta_{va}(\minus)},
	\f{\bmeta_{va}(\minus)}{2-\bmeta_{va}(\minus)},
	\f{\bmeta_{va}(\free)}{2-\bmeta_{va}(\minus)},
	\f{\bmeta_{va}(\plus)}{2-\bmeta_{va}(\minus)}
	\bigg)\,,\\
	\Big(\hq_{av}(\red),\hq_{av}(\yel),
		\hq_{av}(\grn),\hq_{av}(\blu)\Big)
	&=\bigg(
	\f{\bhu_{av}(\plus)}{3-2\bhu_{av}(\plus)},
	\f{\bhu_{av}(\free)}{3-2\bhu_{av}(\plus)},
	\f{\bhu_{av}(\free)}{3-2\bhu_{av}(\plus)},
	\f{\bhu_{av}(\free)}{3-2\bhu_{av}(\plus)}
	\bigg)\,.
	\label{e:ctov.hq.given.bhu}
	\end{align}
Recall also from \S\ref{ss:wp.recursions} that $\bhu_{av}$ can be recursively computed as (cf.\ \eqref{e:first.defn.of.bhu})
	\beq\label{e:clauses.in.terms.of.depth.one.vars}
	\Big(\bhu_{av}(\plus),
	\bhu_{av}(\free)\Big)
	=\bigg(
	\prod_{u\in (\ptree \cap \pd a) \setminus v}
	\bmeta_{ua}(\minus),
	1-\prod_{u\in (\ptree \cap \pd a) \setminus v}
	\bmeta_{ua}(\minus)
	\bigg)\,.
	\eeq
In view of these relations, for $\ptree$ to be nice, it suffices to have
	\begin{align}
	\label{e:final.bound.up.from.depth.one}
	\max\bigg\{
	\bigg|
	\bmeta(\ptree)_{ua}(\minus)-\f12\bigg|
	\,\bigg| \,
	\begin{array}{c}
	a\in\ptree \cap \pd v,\\
	u\in\pd a\setminus v
	\end{array}
	\bigg\}&\le \f1{2^{k/5}}\,,\\
	\label{e:final.bound.out.from.root}
	\max
	\bigg\{
	\sum_{x\in\set{\plus,\minus,\free}}
	\bigg|2^{k\Ind{x=\free}}
	\bmeta(\ptree)_{va}(x)-\f12\bigg|
	\,\bigg|\, a\in\ptree \cap \pd v
	\bigg\}
	&\le \f1{2^{k/5}}\,.
	\end{align}
Indeed, substituting
\eqref{e:final.bound.up.from.depth.one}
into \eqref{e:clauses.in.terms.of.depth.one.vars}
and \eqref{e:ctov.hq.given.bhu}
shows that $\hq=\hqstar(\ptree)$ satisfies condition
\eqref{e:nice.clause.message.cond};
while substituting
\eqref{e:final.bound.out.from.root}
into \eqref{e:vtoc.q.given.eta}
shows that $q=\dqstar(\ptree)$ satisfies condition
\eqref{e:nice.var.message.cond}.

It remains to bound, on the event $\DEG'$, the probability for \eqref{e:final.bound.up.from.depth.one} and \eqref{e:final.bound.out.from.root} to hold for every $\CC$-modification $\ptree$ of the original random tree $\tree\sim\uPGW(T)$. For any such $\ptree$, for all $a\in\ptree\cap\pd v$ and all $u\in\pd a\setminus v$, the subtree $\ptree_{ua}$ is a $\CC$-perturbation of $\tree_{ua}$, so Lemma~\ref{l:perturb.below.variables} applies. For all $a\in\ptree\cap\pd v$, the subtree $\ptree_{va}$ is a $\CC$-perturbation of $\tree_{va}$, so Lemma~\ref{l:perturb.below.variables} again applies. It follows by a simple union bound that for $\tree\sim\uPGW(T)$, with probability lower bounded by $1-\exp(-\Omega(k2^{k/4}))$, every $\CC$-modification $\ptree$ of $\tree$ satisfies the degree condition \eqref{e:nice.degree.condition} and the message conditions \eqref{e:final.bound.up.from.depth.one} and \eqref{e:final.bound.out.from.root}, and hence is nice. This concludes the proof.
\end{proof}

\subsection{Stability in the Galton--Watson tree}
\label{ss:pgw.stability}

The main goal of this section is to prove the following:

\begin{ppn}\label{p:one.stable}
For $R$ exceeding a large absolute constant, we have
	\[\uPGW\Big(
	\textup{$\vrt$ not $1$-stable}\Big)
	\le \f1{\exp(2^{k/20} R)}\,,
	\]
where the $1$-stable property is given by Definition~\ref{d:j.stable}.
\end{ppn}

The next few results (Corollaries~\ref{c:pth.moment.of.eta}--\ref{c:exp.minus.p.times.X}, and Lemma~\ref{l:d.to.p.times.theta.to.d})
give some preliminary estimates, 
related to the concentration bounds from \S\ref{ss:technical.bounds.on.dist.rec}. They will be used in the proof of Lemma~\ref{l:stability} below.

\begin{cor}\label{c:pth.moment.of.eta}
Let $\mu$ be any probability measure over $\eta\in[0,1)$ that satisfies condition~\eqref{i:mu.ell.upper.tail.bound} from Lemma~\ref{l:mu.ell.tail.bounds}. If $p=2^{k\delta}$ for any $1/100 \le \delta \le 1/9$, then we have the $p$-th moment bound
	\[
	\int \eta^p \,d\mu(\eta)
	\le \f1{2^p}\bigg[
	1 + \f{k}{2^{k(1/4-2\delta)}}\bigg]\,.
	\]

\begin{proof} 
Write $\delta'\equiv 1/4-\delta$. We can decompose
	\[
	E(p) \equiv
	\int \eta^p \,d\mu(\eta)
	\le
	\f1{2^p}
	\bigg(
	1+\f{\log k}{2^{k\delta'}}
	\bigg)^p
	+ \bigg(\f23\bigg)^p
	\mu\bigg( \eta \ge 
	\f12 \bigg(1+\f{\log k}{2^{k\delta'}}\bigg)
	\bigg)
	+\mu\bigg( \eta \ge \f23\bigg)\,.
	\]
Note that $\delta'>\delta$, so $(p\log k)/2^{k\delta'}$ is small. Combining with \eqref{i:mu.ell.upper.tail.bound} gives
	\[
	E(p)
	\le
	\f1{2^p}
	\bigg\{ 1 +\f{O(p\log k)}{2^{k\delta'}}\bigg\}
	+\f{(2/3)^p}
		{\exp(\Omega( 2^{k/4} 2^{-k\delta'}\log k ))}
	+\f1{\exp(\Omega(2^{k/4}))}\,.
	\]
Substituting $p=2^{k\delta}$ and $\delta'=1/4-\delta$ into the above gives
	\[E(p)
	\le\f1{2^p}\bigg[
	1 + \f{O(\log k)}{2^{k(1/4-2\delta)}}
	+\f{\exp\{2^{k\delta}\log(4/3)\}}
		{\exp(\Omega( 2^{k\delta}\log k)
			)}
	+\f{\exp\{ 2^{k\delta}\log2 \}}{\exp(\Omega(2^{k/4}))}
	\bigg]
	\le \f1{2^p}\bigg[
	1 + \f{O(\log k)}{2^{k(1/4-2\delta)}}\bigg]\,,
	\]
which concludes the proof.
\end{proof}
\end{cor}

\begin{cor}\label{c:stability.base.case}
Let $\mu$ be any probability measure over $\eta\in[0,1)$ that satisfies conditions~\eqref{i:mu.ell.upper.tail.bound}~and~\eqref{i:mu.ell.left.tail.bound} from Lemma~\ref{l:mu.ell.tail.bounds}. If $p=2^{k\delta}$ for $1/100\le\delta\le 1/9$, then
	\[\int 
	\bigg|\f{\eta-1/2}{\eta+1/2}\bigg|^p\,d\mu(\eta)
	\le\f{k^{2p}}{2^{kp(1/4-\delta)}}\,.\]

\begin{proof} It follows from \eqref{i:mu.ell.upper.tail.bound}~and~\eqref{i:mu.ell.left.tail.bound} that
for any $\delta < \delta' < 1/4$ we have
	\[
	\int 
	\bigg|\f{\eta-1/2}{\eta+1/2}\bigg|^p\,d\mu(\eta)
	\le
	\f{O(1)}{2^{kp \delta'}}
	+ \f1{\exp( 2^{k/4} 2^{-k\delta'} )}\,.
	\]
The claim follows by setting $\delta' = 1/4-\delta -(\log k)/(k\log2)$.
\end{proof}
\end{cor}

\begin{cor}\label{c:exp.minus.p.times.X}
Let $\mu$ be any probability measure over $\eta\in[0,1)$
that satisfies condition~\eqref{i:mu.ell.left.tail.bound} from Lemma~\ref{l:mu.ell.tail.bounds}. Write $\mu^{\otimes}$ for the law of a sequence $\ueta'\equiv(\eta_j)_{j\ge1}$ of i.i.d.\ samples from $\mu$. If $p=2^{k\delta}$ for any $1/100\le\delta\le 1/9$, then
	\[\int\bigg(
	1-\prod_{j=1}^{k-1} \eta_j\bigg)^p
	\,d\mu^{\otimes}(\ueta')
	\le \exp\bigg\{ -\f{p}{2^{k-1}}\bigg[1-
	\f{k^3}{2^{k/4}}
	\bigg]\bigg\}\,.\]

\begin{proof}
Write $\P$ and $\E$ for probability and expectation over the law of
	\[
	X
	\equiv -\log \bigg(
	1-\prod_{j=1}^{k-1} \eta_j\bigg)\,.
	\]
The quantity of interest is then $\E(e^{-pX})$. 
Recall that Lemma~\ref{l:X.tail.bounds}\eqref{i:X.left.tail} gives
	\[
	\P\bigg(
	X \le \bigg( \f{2^{-\xi}}{2}\bigg)^{k-1}
	\bigg)
	\le \f1{\exp\{\Omega(\xi 2^{k/4})\}}
	\]
for all $k/2^{k/4} \le\xi\le1$. It follows that
	\begin{align*}
	\E(e^{-pX})
	&\le
	\exp\bigg\{
	-\f{p}{2^{k-1}}
	\bigg[1 - \f{k^2\log k}{2^{k/4}}\bigg]
	\bigg\}
	+\P\bigg(
	X \le \f1{2^{k-1}}\bigg[1 - \f{k^2\log k}{2^{k/4}}\bigg]
	\bigg)\\
	&\le \exp\bigg\{
	-\f{p}{2^{k-1}}
	\bigg[1 - \f{k^2\log k}{2^{k/4}}\bigg]
	\bigg\} + \f1{\exp\{\Omega(k\log k )\}}
	\le
	\exp\bigg\{
	-\f{p}{2^{k-1}}
	\bigg[1 - \f{O(k^2\log k)}{2^{k/4}}
		\bigg]
	\bigg\}\,,
	\end{align*}
and this implies the claim.
\end{proof}
\end{cor}

\begin{lem}\label{l:d.to.p.times.theta.to.d}
Let $D$ be a 
$\Pois(\alpha k/2)$ random variable,
and let $\E$ denote expectation over the law of $D$.
If $p=2^{k\delta}$ for $1/100\le\delta\le 1/9$, then
	\[\E\bigg\{
	D^p
	\exp\bigg\{
	-\f{Dp}{2^{k-1}}
	\bigg[
	1-\f{k^3}{2^{k/4}}
	\bigg]
	\bigg\}
	\bigg\}
	\le \bigg(\f{3k}{8}\bigg)^p\,.\]

\begin{proof}
The quantity of interest can be written as
$\E[D^p /e^{\gamma D}]$ for
	\[
	\gamma
	\equiv \f{p}{2^{k-1}}\bigg[
	1-\f{k^3}{2^{k/4}}
	\bigg]\,.
	\]
We can check by differentiation
that $d^p / e^{\gamma d}$ is
decreasing with respect to $d$ for all
	\[
	d > \f{p}{\gamma}
	= 2^{k-1} \bigg/ \bigg[
	1-\f{k^3}{2^{k/4}}
	\bigg]\,.\]
For all $\alpha$ in the regime
\eqref{e:alpha.regime}, the mean 
$\E D=\alpha k/2$ is much larger than $p/\gamma$.
As in \eqref{e:alpha.regime} let $\aubd\equiv 2^k\log2$.
Then
	\begin{align*}
	\E\bigg[ \f{D^p}{\exp(\gamma D)}
	; D \ge \f{k\aubd}{2}
	\bigg]
	&\le 
	\bigg(\f{k\aubd}{2}\bigg)^p
	\exp\bigg\{-\f{\gamma k\aubd}{2}\bigg\}\\
	&= \bigg(
	\f{k 2^k \log2}{2}
	\exp\bigg\{ -k\log2
		\bigg[1-\f{k^3}{2^{k/4}}\bigg] \bigg\}
	\bigg)^p
	\le O(1) \bigg(
		\f{k\log2}{2}\bigg)^p\,.
	\end{align*}
On the other hand, we can use the Poisson moment-generating function to bound
	\begin{align*}
	\E\bigg[ \f{D^p}{\exp(\gamma D)}
	; D \le \f{k\aubd}{2}
	\bigg]
	&\le
	\bigg(\f{k\aubd}{2}\bigg)^p
	\E\bigg[\f1{ \exp(\gamma D)}\bigg]\\
	&=\bigg(\f{k\aubd}{2}\bigg)^p
	\exp\bigg\{- \f{\alpha k}{2}
		\bigg(1-\f1{e^\gamma}\bigg)\bigg\}
	\le 
	O(1) \bigg( \f{k\log2}{2}\bigg)^p\,.
	\end{align*}
Combining the bounds gives the claim.
\end{proof}
\end{lem}

\begin{lem}\label{l:d.to.p.times.rare.event}
Let $D$ be a $\Pois(\alpha k/2)$ random variable, and let $\bm{F}_i$ be events such that
whenever $D \ge \alpha k/4$, we have
$\P(\bm{F}_i \,|\, D ) \le \exp(-2^{k/4})$
for all $1\le i\le D$. If $p=2^{k\delta}$ for $1/100\le\delta\le 1/9$, then
	\[\E\bigg[
	D^{p-1}\sum_{i=1}^{D} \mathbf{1}_{\bm{F}_i}
	\bigg] \le
	\f1{\exp(2^{k/5})}
	\,.\]

\begin{proof} It follows from the Poisson Chernoff bound
\eqref{e:poisson.chernoff.bound} that
	\[
	\E\bigg[
	D^{p-1}\sum_{i=1}^{D} \mathbf{1}_{\bm{F}_i}
	; D \le \f{\alpha k}{4}
	\bigg]
	\le \bigg( \f{\alpha k}{4}\bigg)^p
	\P\bigg(D \le \f{\alpha k}{4}\bigg)
	\le \bigg( \f{\alpha k}{4}\bigg)^p
	\f1{\exp(\Omega(k 2^k) )}
	\le \f1{\exp(\Omega(k 2^k) )}\,.
	\]
We can use integration by parts
and \eqref{e:poisson.chernoff.bound} to bound
	\begin{align*}
	\E\bigg(D^p ; D \ge3^k\bigg)
	&\le 3^{kp} \P(D\ge3^k)
	+ \int_{3^k}^\infty pt^{p-1} \P(D\ge t)\,dt
	\le \f{3^{kp}}{\exp(\Omega(k3^k))}
	+ p \int_{3^k}^\infty 
	\f{t^{p-1} }{\exp(\Omega(kt)) }
		\,dt\\
	&\le \f{3^{kp}}{\exp(\Omega(k3^k))}
	+ \f{p}{\exp(\Omega(k3^k))}
	\le \f1{\exp(\Omega(k3^k))}\,,
	\end{align*}
so we have $\E(D^p) \le O(3^{kp})$. It follows from the assumption on the $\bm{F}_i$ that
	\[
	\E\bigg[
	D^{p-1}\sum_{i=1}^{D} \mathbf{1}_{\bm{F}_i}
	; D \ge \f{\alpha k}{4}
	\bigg]
	\le \f{\E(D^p)}{\exp(2^{k/4})}
	\le \f1{\exp(\Omega(2^{k/4}))}\,,
	\]
which implies the claim.
\end{proof}
\end{lem}

\begin{lem} \label{l:wk.rademacher.bound}
Let $Z,Z_i$ be i.i.d.\ random variables
with symmetric distribution
(meaning that $Z$ is equidistributed as $-Z$).
Then, for any positive integer $d$ and positive even integer $p$,
	\[
	\E\bigg( \sum_{i=1}^d Z_i\bigg)^p
	\le O(1) \bigg( \f{dp}{e}\bigg)^{p/2}
	\E(|Z|^p)\,.
	\]

\begin{proof}
Writing $\vec{i}\equiv(i_1,\ldots,i_p)$ for elements of $[d]^p$, we expand
	\[
	\E\bigg( \sum_{i=1}^d Z_i\bigg)^p
	= \sum_{\vec{i}\in[d]^p}
	\E[ Z_{i_1} \cdots Z_{i_p}]\,.
	\]
On the right-hand side, $\E[ Z_{i_1} \cdots Z_{i_p}]$ is zero unless every index appears an even number of times. The number of choices for $\vec{i}$ for which this holds is upper bounded by $d^{p/2}(p-1)!!$, where
$(p-1)!!$ is the number of matchings on $p$ elements. By Stirling's formula,
	\[
	d^{p/2}(p-1)!!
	=\f{d^{p/2} p!}{2^{p/2} (p/2)!}
	\le O(1) \bigg( \f{dp}{e}\bigg)^{p/2}\,.
	\]
By Jensen's inequality, for any $\vec{i}$ we have
$\E[ Z_{i_1} \cdots Z_{i_p}] \le \E(|Z|^p)$. Combining these bounds gives the claim.
\end{proof}
\end{lem}

The next three lemmas record some simple (deterministic) bounds, based on elementary calculus manipulations, which will also be used in the proof of 
Lemma~\ref{l:stability} below. For $0\le x\le 1$ and $y\in\mathbb{R}$, define the function
	\beq\label{e:def.F.x.y}
	F(x,y)\equiv \f{(1-x)e^y}{1+(1-x)e^y}.
	\eeq
and note that $F$ takes values in $[0,1)$. Write $\nabla F \equiv (F_x,F_y)\equiv (\pd F/\pd x,\pd F/\pd y)$.

\begin{lem}\label{l:F.bounds.calculus}
Let $F$ be as defined by \eqref{e:def.F.x.y}. If $0\le x_i \le 1$ and $y_i\in\mathbb{R}$ with $x_i e^{y_i}\le1$, then
	\begin{align}\label{e:F.bdd.derivs}
	\bigg|F(x_1,y_1)-F(x_2,y_2)\bigg|
	&\le \Big|x_1-x_2\Big| + \Big|y_1-y_2\Big|\,,\\
	\label{e:F.second.derivs}
	\Bigg|F(x_1,y_1)-F(x_2,y_2)
	-\bigg\langle
	\nabla F(x_1,y_1),
	\begin{pmatrix}
	x_1-x_2\\
	y_1-y_2
	\end{pmatrix}
	\bigg\rangle\Bigg|
	&\le c\bigg\{
	(x_1-x_2)^2+(y_1-y_2)^2\bigg\}\,,
	\end{align}
where $c$ is an absolute constant.

\begin{proof}
Write $F_{xx},F_{xy},F_{yy}$ for the second-order partial derivatives of $F$. It is straightforward to verify that for $0\le x\le 1$ and all $y\in\mathbb{R}$ we have
$|F_y|\le1$ and $|F_{yy}|\le1$. Further, under the additional restriction that $xe^y\le1$, we have
$|F_x|\le1$, $|F_{xx}|\le2$, and $|F_{xy}|\le1$. The bounds \eqref{e:F.bdd.derivs} and \eqref{e:F.second.derivs} directly follow. 
\end{proof}
\end{lem}

\begin{lem}\label{l:F.diff.ratio.bound}
Let $F$ be as defined by \eqref{e:def.F.x.y}. For any $0\le x_i\le1$ and $y_i\in\mathbb{R}$,
	\beq\label{e:F.ratios.bd}
	\bigg|\f{F(x_1,y_1)-F(x_2,y_2)}
	{F(x_1,y_1)+F(x_2,y_2)}\bigg|
	\le \f{|x_1-x_2|}{2-x_1-x_2}
		+ \f{|y_1-y_2|}{2}\,.\eeq
(Unlike Lemma~\ref{l:F.bounds.calculus}, this does not require $x_i e^{y_i}\le1$.)

\begin{proof} Note that the function $G(z)\equiv z/(1+z)$ satisfies
	\beq\label{e:G.diff.ratio}
	\f{|G(w)-G(z)|}{G(w)+G(z)}
	=\f{|w-z|}{w+z+2wz}
	\le\f{|w-z|}{w+z}
	\eeq
for any $w,z\ge0$. It follows from \eqref{e:G.diff.ratio} that for any $0\le x_i\le1$ and any fixed $y\in\mathbb{R}$,
	\beq\label{e:F.ratio.diff.fixed.y}
	\f{|F(x_2,y)-F(x_1,y)|}
		{F(x_2,y)+F(x_1,y)}
	=
	\f{|G((1-x_2)e^y)-G((1-x_1)e^y)|}
		{G((1-x_2)e^y)+G((1-x_1)e^y)}
	\le \f{|x_2-x_1|}{2-x_1-x_2}\,.\eeq
A further consequence of 
\eqref{e:G.diff.ratio} is that for any $w,z\in\mathbb{R}$,
	\[\f{|G(e^w)-G(e^z)|}{G(e^w)+G(e^z)}
	\le
	\f{|e^w-e^z|}{e^w+e^z}
	=\Th\bigg(\f{|w-z|}{2}\bigg)
	\le \f{|w-z|}{2}\,.
	\]
It follows from this that for any fixed
$0\le x\le 1$ and any $y_i\in\mathbb{R}$,
	\beq\label{e:F.ratio.diff.fixed.x}
	\f{|F(x,y_1)-F(x,y_2)|}{F(x,y_1)+F(x,y_2)}
	= \f{|G((1-x)e^{y_2})-G((1-x)e^{y_1})|}
		{G((1-x)e^{y_2})+G((1-x)e^{y_1})}
	\le \f{|y_1-y_2|}2\,.\eeq
Combining \eqref{e:F.ratio.diff.fixed.y}~and~\eqref{e:F.ratio.diff.fixed.x} gives the claim.
\end{proof}
\end{lem}

\begin{lem}\label{l:ratio.a.b.bound}
For any positive numbers $a_j$ and $b_j$,
	\beq\label{e:ratio.a.b.bound}
	\bigg|\f{\prod_{j=1}^k a_j-\prod_{j=1}^k b_j}
	{\prod_{j=1}^k a_j+\prod_{j=1}^k b_j}
	\bigg|
	\le \sum_{\emptyset\subsetneq J\subseteq [k]}
		2^{|J|}
		\prod_{j\in J}
		\bigg|\f{a_j-b_j}{a_j + b_j}\bigg|\,.
	\eeq

\begin{proof}
Partition the indices $[k]\equiv\set{1,\ldots,k}$
into $A\equiv\set{j : a_j \ge b_j}$
and $B\equiv[k]\setminus A$. Write $\delta\equiv a-b$. Then
	\begin{align*}
	N
	&\equiv \prod_{j=1}^k a_j-\prod_{j=1}^k b_j
	= \prod_{j\in A} (b_j+\delta_j)\prod_{j\in B} a_j
	 -\prod_{j\in B} (a_j-\delta_j) \prod_{j\in A} b_j\\
	&=\sum_{\emptyset\subsetneq J\subseteq A}
		\prod_{j\in J} \delta_j
		\prod_{j\in A\setminus J} b_j
		\prod_{j\in B} a_j
	-\sum_{\emptyset\subsetneq J\subseteq B}
		\prod_{j\in J} (-\delta_j)
		\prod_{j\in B\setminus J} a_j
		\prod_{j\in A} b_j\,.
	\end{align*}
It follows using the definition of $A$ and $B$ that
	\[
	|N|
	\le
	\sum_{\substack{\emptyset\subsetneq J\subseteq [k],\\
		J\subseteq A\textup{ or }J\subseteq B
		}}
		\prod_{j\in J} |\delta_j|
		\prod_{j\in[k]\setminus J} \min\set{a_j,b_j}\,,
	\]
On the other hand, for any $J\subseteq A$, we have
	\[
	D\equiv
	\prod_{j=1}^k a_j + \prod_{j=1}^k b_j
	\ge \prod_{j=1}^k a_j
	\ge \prod_{j\in J} \max\set{a_j,b_j}
	\prod_{j\in [k]\setminus J} \min\set{a_j,b_j}
	\ge
	\prod_{j\in J} \f{a_j+b_j}{2}
	\prod_{j\in [k]\setminus J} \min\set{a_j,b_j}\,,
	\]
and the same bound holds for any $J\subseteq B$. Combining the last two bounds gives
	\[
	\f{|N|}{D}
	\le
	\sum_{\substack{\emptyset\subsetneq J\subseteq [k],\\
		J\subseteq A\textup{ or }J\subseteq B
		}} \prod_{j\in J} 
		\f{2|\delta_j|}{a_j+b_j}\,,
	\]
which implies the claim.
\end{proof}
\end{lem}

Now recall from Definition~\ref{d:PGW.with.clauses.one.less} the measure 
$\PGW_\EPS$. Let $\tree\equiv\tree_{\vrt\crt}\sim\PGW_\EPS$, 
and (as in \eqref{e:coupling.of.all.eta.ell}) let
	\beq\label{e:coupling.eta.with.eps.param}
	(\bmeta^\ell)_{\ell\ge0}
	\equiv \bigg(\FF_\ell(
	\tree_{\vrt\crt} )\bigg)_{\ell\ge0}\,.\eeq
The next lemma gives the main technical estimate which will be used to prove stability in the Galton--Watson tree.

\begin{lem} \label{l:stability}
Write $\P_\EPS$ for the law of the random sequence defined by \eqref{e:coupling.eta.with.eps.param}, and write $\E_\EPS$ for expectation with respect to $\P_\EPS$. Let $\eta^\ell\equiv\bmeta^\ell(\minus)$. For $p= 2 \lceil 2^{k/10}\rceil$ we have the bound
	\[\E_\EPS\Bigg[\bigg|
	\f{\eta^{\ell+1}-\eta^\ell}{\eta^{\ell+1}+\eta^\ell}
	\bigg|^p\Bigg]
	\le
	\f1{2^{kp/7} 2^{2kp\ell/5}}
	\le
	\f1{2^{kp(\ell+1)/7}}\]
for all $\ell\ge0$. This holds for any $0\le\EPS\le1$.

\begin{proof} We will prove the bound by induction, starting from the base case $\ell=0$:
by definition, we have $\eta^0=1/2$ (with probability one), and so
Corollary~\ref{c:stability.base.case} gives
	\[\E_\EPS\bigg[\bigg|
	\f{\eta^1-\eta^0}{\eta^1+\eta^0}
	\bigg|^p\bigg]
	=\E_\EPS\bigg[\bigg|
	\f{\eta^1-1/2}{\eta^1+1/2}
	\bigg|^p\bigg]
	\le \f1{2^{kp/7}}\,.
	\]
Now suppose inductively that for some $\ell\ge1$ we have
	\beq\label{e:stability.inductive}
	\E_\EPS\bigg[
	\bigg|\f{\eta^\ell-\eta^{\ell-1}}{2}\bigg|^p
	\bigg]
	\le\E_\EPS\bigg[\bigg|
	\f{\eta^\ell-\eta^{\ell-1}}{\eta^\ell+\eta^{\ell-1}}
	\bigg|^p\bigg]
	\le
	I(\ell)^p
	\equiv
	\f1{2^{kp/7} 2^{2kp\ell/5}}
	\,.\eeq
Let $A\equiv(H,h)$ denote a random variable with the same law as the pair $(\eta^\ell,\eta^{\ell-1})$. Let
	\[
	\vec{A}
	\equiv\Bigg( 
	A_j \equiv \begin{pmatrix}
		H_j \\ h_j
		\end{pmatrix},
	A^\plus_{ij} \equiv \begin{pmatrix}
		H^\plus_{ij} \\ h^\plus_{ij}
		\end{pmatrix},
	A^\minus_{ij}
		\equiv \begin{pmatrix}
		H^\minus_{ij} \\ h^\minus_{ij}
		\end{pmatrix}\Bigg)_{i,j\ge1}
	\]
be an array of i.i.d.\ copies of $A$. Let
$K$ denote a random variable which takes value $k$ with chance $1-\EPS$, and takes value $k-1$ with chance $\EPS$. Let $\vec{K} \equiv (K^\plus_{ij},K^\minus_{ij})_{i,j\ge1}$ be an array of i.i.d.\ copies of $K$, and let
	\[
	S^\PM_i
	\equiv 1-\prod_{j=1}^{K^\PM_i-1} H^\PM_{ij}
	\,,\quad
	s^\PM_i
	\equiv 1-\prod_{j=1}^{K^\PM_i-1} h^\PM_{ij}.
	\]
Note that by construction the $H$ and $h$ random variables lie in $[0,1)$ almost surely, so the $S$ and $s$ random variables lie in $(0,1]$ almost surely. Let $\ud\equiv(d^\plus,d^\minus)\sim\POpm$, and use this to define the random variables
	\beq\label{e:stability.notations}
	\Pi^\PM
	\equiv \prod_{i=1}^{d^\PM} S^\PM_i\,,\quad
	\pi^\PM
	\equiv \prod_{i=1}^{d^\PM} s^\PM_i\,,\quad
	\Sigma
	\equiv\log \f{\Pi^\plus}{\Pi^\minus}\,,\quad
	\sigma
	\equiv\log \f{\pi^\plus}{\pi^\minus}\,.\eeq
For the remainder of the proof we will abbreviate
$S_i\equiv S^\minus_i$,
$s_i\equiv s^\minus_i$,
$\Pi \equiv \Pi^\minus$, $\pi\equiv \pi^\minus$. 
With this notation, the pair $(\eta^{\ell+1},\eta^\ell)$ is equidistributed as
	\[\bigg(
	\f{(1-\Pi)\exp(\Sigma)}{1+(1-\Pi)\exp(\Sigma)},
	\f{(1-\pi)\exp(\sigma)}{1+(1-\pi)\exp(\sigma)}
	\bigg)
	=\bigg( F(\Pi,\Sigma),F(\pi,\sigma)\bigg)
	\,,
	\]
for $F$ as defined by \eqref{e:def.F.x.y}. To prove the result, it suffices to bound $\E(J^p)$ for
	\[
	J^p
	\equiv 
	\bigg|
	\f{ F(\Pi,\Sigma)-F(\pi,\sigma)}
		{ F(\Pi,\Sigma)+F(\pi,\sigma)}\bigg|^p
	\le\bigg(\f{2|\Pi-\pi|}{2-\Pi-\pi}\bigg)^p
	 + |\Sigma-\sigma|^p\,,\]
where the last bound follows by Lemma~\ref{l:F.diff.ratio.bound}.
Since each of $\Pi,\pi$ is a product over $d^\minus$ terms, we can decompose their difference as a telescoping product:
	\[
	\Pi-\pi
	=\prod_{i=1}^{d^\minus} S_i
	-\prod_{i=1}^{d^\minus} s_i
	=\sum_{i=1}^{d^\minus} 
		\Pi[i]
		(S_i-s_i)\,,\quad
	\Pi[i]\equiv
		\prod_{j=1}^{i-1} S_j
		\prod_{j=i+1}^{d^\minus} s_j\,.
	\]
It follows from H\"older's inequality that
for any $\mathbf{u}\in\mathbb{R}^d$,
$|(\mathbf{1},\mathbf{u})|^p
\le d^{p-1} (\|\mathbf{u}\|_p)^p$. This implies
	\beq\label{e:diff.Pi.telescope}
	|\Pi-\pi|^p
	\le
	(d^\minus)^{p-1}
	\sum_{i=1}^{d^\minus}
	\Pi[i]^p |S_i-s_i|^p\,.
	\eeq
Note that $0\le \Pi,\pi\le1$, and $\Pi\pi \le \Pi[i]$ for any $i$. Let $\bm{F}_i$ be the event that $\Pi[i]\ge 1/4$.
If $\bm{F}_i$ does not occur, then we must have
either $\Pi\le1/2$ or $\Pi\le1/2$, therefore $2-\Pi-\pi \ge 1/2$. It follows that
	\beq\label{e:stability.decomp}
	J^p\le
	\underbrace{\Bigg[
	4^p (d^\minus)^{p-1}
	\sum_{i=1}^{d^\minus}
	\Pi[i]^p |S_i-s_i|^p\Bigg] }_{J_1(p)}
	+\underbrace{\Bigg[ 2^p (d^\minus)^{p-1}
	\sum_{i=1}^{d^\minus}
	\mathbf{1}_{\bm{F}_i}
	\f{\Pi[i]^p |S_i-s_i|^p}{(2-\Pi-\pi)^p}
	\Bigg]}_{J_2(p)}
	+|\Sigma-\sigma|^p\,.\eeq
Similarly to \eqref{e:diff.Pi.telescope}, we can expand $S-s \equiv S_i-s_i$ as a telescoping sum, then use H\"older's inequality to bound
	\beq\label{e:telescope.def.U.j}
	|S-s|^p
	=\bigg|
	\prod_{j=1}^{K-1} H_j
	-\prod_{j=1}^{K-1} h_j
	\bigg|^p
	\le k^{p-1}
	\sum_{j=1}^{K-1} U[j]^p
		|H_j-h_j|^p\,,\quad
		U[j]\equiv \prod_{t=1}^{j-1} H_t
		\prod_{t=j+1}^{K-1} h_t\,.
	\eeq
It follows by considerations of conditional independence that
	\[
	\E J_1(p)
	\le (4k)^p
	\E\bigg[ (d^\minus)^p
	(\max\set{\E(S^p),\E(s^p)})^{d^\minus-1}\bigg]
	\bigg(\max\set{ \E(H^p),\E(h^p)}\bigg)^{k-3}
	\E(|H-h|^p)\,.
	\]
Both $\E(S^p)$ and $\E(s^p)$ satisfy the bound from Corollary~\ref{c:exp.minus.p.times.X}; and combining with Lemma~\ref{l:d.to.p.times.theta.to.d} gives a bound on the first expectation on the right-hand side above. Both $\E(H^p)$ and $\E(h^p)$ satisfy the bound from Corollary~\ref{c:pth.moment.of.eta}. The last factor
$\E(|H-h|^p)$ is bounded by
the inductive hypothesis \eqref{e:stability.inductive}. Altogether it gives
	\beq\label{e:stability.first.finalbound}
	\E J_1(p)
	\le O(1) (4k)^p
	\bigg(\f{3k}{8}\bigg)^p
	\f{\E (|H-h|^p)}{2^{p(k-3)}}
	\le
	\bigg( \f{k^3}{2^k}\bigg)^p
	\f1{2^{kp/7} 2^{2kp\ell/5}} 
	=\bigg( \f{k^3 I(\ell)}{2^k} \bigg)^p
	\,.\eeq
Turning to $J_2(p)$, we note that (very crudely) we have
$0\le \Pi[i]\le1$ and $2-\Pi-\pi \ge 2-S_i-s_i$
for any $1\le i\le d^\minus$. Therefore, again with considerations of conditional independence, we have
	\beq\label{e:J.2.p}
	\E J_2(p)
	\le 2^p
	\E\bigg[
	(d^\minus)^{p-1}
	\sum_{i=1}^{d^\minus} \mathbf{1}_{\bm{F}_i}
	\bigg]
	\E\bigg[ \bigg|\f{S-s}{2-S-s}\bigg|^p\bigg]
	\,.\eeq
Recall that $\bm{F}_i$ is the event that $\Pi[i]\ge1/4$, or equivalently that
	\[
	\Sigma[i]
	\equiv \log\f1{\Pi[i]}
	= \sum_{j=1}^{i-1} \log \f1{S_i}
		+\sum_{j=i+1}^{d^\minus} \log \f1{s_i}
	\le \log 4\,.
	\]
By essentially the same argument as for
\eqref{e:left.tail.Sigma}, as long as $d^\minus\ge \alpha k/4$ we have $\P(\bm{F}_i\,|\,d^\minus)\le \exp(-\Omega(k2^{k/4}))$. We can then apply Lemma~\ref{l:d.to.p.times.rare.event} to bound the first expectation in \eqref{e:J.2.p}:
	\beq\label{e:stability.rare.event.bnd}
	\E\bigg[
	(d^\minus)^{p-1}
	\sum_{i=1}^{d^\minus} \mathbf{1}_{\bm{F}_i}
	\bigg]
	\le \f1{\exp(2^{k/5})}\,.
	\eeq
For the second expectation
in \eqref{e:J.2.p}, Lemma~\ref{l:ratio.a.b.bound} gives
	\[
	\E\bigg[ \bigg|\f{S-s}{2-S-s}\bigg|^p\bigg]
	=
	\E\Bigg[ \Bigg|
	\f{\prod_{j=1}^{K-1} H_j-\prod_{j=1}^{K-1} h_j}
		{\prod_{j=1}^{K-1} H_j+\prod_{j=1}^{K-1} h_j}
	\Bigg|^p\Bigg]
	\le
	\E \Bigg[\bigg(
	\sum_{\emptyset \subsetneq J\subseteq [k-1]}
	2^{|J|}
	\prod_{j\in J} 
	\bigg|\f{H_j-h_j}{H_j+h_j}\bigg|\bigg)^p
	\Bigg]\,.
	\]
Let $\vec{J}\equiv(J_1,\ldots,J_p)$ denote any $p$-tuple of nonempty subsets of $[k-1]$. We abbreviate
$|\vec{J}| \equiv |J_1| + \ldots + |J_p| \ge p$.
For any $j\in[k-1]$, let
$n_j(\vec{J})$ denote the number of occurrences of $j$ in the sets $J_1,\ldots,J_p$:
	\[
	n_j(\vec{J})
	=\sum_{t=1}^p
	\Ind{j \in J_t}\,,\quad
	\sum_{j=1}^{k-1} n_j(\vec{J})
	= \sum_{t=1}^p |J_t| = |\vec{J}| \ge p\,.
	\]
With this notation, we can bound
	\[
	\E\bigg[ \bigg|\f{S-s}{2-S-s}\bigg|^p\bigg]
	\le\sum_{\vec{J}}
	2^{|\vec{J}|}
	\prod_{j=1}^{k-1}
	\E\Bigg[
	\bigg|\f{H_j-h_j}{H_j+h_j}\bigg|^{n_j(\vec{J})} \Bigg]
	\le
	\sum_{\vec{J}}
	2^{|\vec{J}|}
	\E\Bigg[
	\bigg|\f{H-h}{H+h}\bigg|^p \Bigg]^{|\vec{J}|/p}\,,
	\]
where the last step is by Jensen's inequality. The number of choices of $\vec{J}\equiv(J_1,\ldots,J_p)$ with $|\vec{J}|=b$ is upper bounded by $k^b$. Combining with the inductive hypothesis \eqref{e:stability.inductive} gives
	\[\E\bigg[ \bigg|\f{S-s}{2-S-s}\bigg|^p\bigg]
	\le
	\sum_{b\ge p}
	\Bigg( 2k
	\E\Bigg[
	\bigg|\f{H-h}{H+h}\bigg|^p \Bigg]^{1/p}
	\Bigg)^b
	\le
	\sum_{b\ge p}
	\Bigg( \f{2k}{2^{k/7} 2^{2k\ell/5}} \Bigg)^b
	\le \f{k^{2p}}{ 2^{kp/7} 2^{2kp\ell/5} }\,.
	\]
Substituting this and \eqref{e:stability.rare.event.bnd}
into \eqref{e:J.2.p} gives
	\beq\label{e:stability.second.finalbound}
	\E J_2(p) \le
	\f{k^{2p}}{\exp(2^{k/5})}
	\f1{ 2^{kp/7} 2^{2kp\ell/5} }
	\le
	\bigg( \f{k^2 I(\ell)}{\exp(2^{k/10})} \bigg)^p
	\,.\eeq
Finally we bound $\E(|\Sigma-\sigma|^p)$, where we recall from \eqref{e:stability.notations} that
	\[
	\Sigma
	=\log \f{\Pi^\plus}{\Pi^\minus}
	= \sum_{i=1}^{d^\plus}\log S^\plus_i
	-\sum_{i=1}^{d^\minus}\log S^\minus_i\,,
	\]
and similarly $\sigma=\log(\pi^\plus/\pi^\minus)$. Let $D\sim\Pois(\alpha k)$, and let $(\mathfrak{s}_i)_{i\ge1}$ be a sequence of i.i.d.\ symmetric random signs. Then, recalling that $p$ is an even integer, we have
	\beq\label{e:diff.Sigma.sigma.pth.mmt}
	\E(|\Sigma-\sigma|^p)
	= \E\Bigg[ \bigg(\sum_{i=1}^D
	\mathfrak{s}_i
	(\log S_i-\log s_i)
	\bigg)^p\Bigg]
	\le O(1) \bigg( \f{dp}{e} \bigg)^{p/2}
		\E\bigg[\bigg|\log \f{S}{s}\bigg|^p\bigg]\,,
	\eeq
where the last step uses Lemma~\ref{l:wk.rademacher.bound}.
We then use a telescoping sum to bound
	\beq\label{e:telescoping.sum.u.h}
	\bigg|\log \f{S}{s}\bigg|
	=\bigg|
	\log \f{1-\prod_{j=1}^{K-1} H_j}
		{1-\prod_{j=1}^{K-1} h_j}
	\bigg|
	\le \sum_{j=1}^{K-1}
	\bigg|\log
		\f{1-U[j] H_j}
		{1-U[j] h_j}\bigg|
	\le
	\sum_{j=1}^{K-1}
	\f{U[j]}{1-U[j]}
	\cdot
	|H_j-h_j|\,,
	\eeq
where the last step uses that the function $f_U(h)=\log(1-Uh)$ has $|(f_U)'(h)| \le U/(1-U)$ for all $U,h\in[0,1]$. In the last expression above, we can further replace $U[j]$ (defined by \eqref{e:telescope.def.U.j}) with
	\[U_\circ[j]
	\equiv \prod_{t=1}^{j-1} H_t
		\prod_{t=j+1}^{k-2} h_t
	\ge U[j]\,.
	\]
It follows by H\"older's inequality, combined with the Cauchy-Schwarz inequality, that
	\beq\label{e:diff.log.S.log.s}
	\E\bigg[\bigg|\log \f{S}{s}\bigg|^p\bigg]
	\le k^p
	\bigg(\max\set{\E(H^{2p}),\E(h^{2p})}\bigg)^{(k-3)/2}
	\max_{j\le k-2}
	\E\bigg[\bigg(\f{1}{1-U_\circ[j]}\bigg)^{2p}\bigg]^{1/2}
	\E(|H-h|^p)
	\eeq
For $x\ge 3$, we have
$\log x + \log(1-1/x) \ge (\log x)/2$. It follows that
	\begin{align*}
	\P\bigg(\f{1}{1-U_\circ[j]}\ge x\bigg)
	&=\P\bigg(U_\circ[j] \ge 1-\f1x\bigg)
	\le\max\bigg\{
	\P\bigg(H\ge 1-\f1x\bigg),
	\P\bigg(h\ge 1-\f1x\bigg)\bigg\}^{k-3}\\
	&\le\max\bigg\{
	\P\bigg( \log\f{H}{1-H} \ge \f{\log x}2 \bigg),
	\P\bigg( \log\f{h}{1-h} \ge \f{\log x}2 \bigg)
	\bigg\}^{k-3}
	\le \f1{\exp( k(\log x) 2^{k/4})}\,,\end{align*}
where the last bound is by \eqref{i:mu.ell.upper.tail.bound} 
from Lemma~\ref{l:mu.ell.tail.bounds}. Integrating this bound gives
	\[\E\Bigg[\bigg(\f{1}{1-U_\circ[j]}\bigg)^{2p}\Bigg]
	\le(3/2) \cdot 3^{2p}
		+ \int_3^\infty \f{pt^{p-1}}{t^{k2^{k/4}}}\,dt
	\le 2 \cdot 3^{2p}\,.
	\]
The other factors in \eqref{e:diff.log.S.log.s}
are controlled by Corollary~\ref{c:pth.moment.of.eta} and \eqref{e:stability.inductive}. Substituting into
\eqref{e:diff.Sigma.sigma.pth.mmt} gives
	\beq\label{e:diff.Sigma.sigma.finalbound}
	\E|\Sigma-\sigma|^p
	\le
	O(1) \bigg( \f{dp}{e} \bigg)^{p/2}
	\f{k^{2p}}{2^{kp}}
	\f1{2^{kp/7} 2^{2kp\ell/5}} 
	\le
	\bigg( \f{k^2 I(\ell)}{2^{9k/20}} \bigg)^p
	\,.\eeq
Substituting
\eqref{e:stability.first.finalbound}, \eqref{e:stability.second.finalbound}, and \eqref{e:diff.Sigma.sigma.finalbound}
into \eqref{e:stability.decomp} gives
	\[
	\E \Bigg[
	\bigg|\f{\eta^{\ell+1}-\eta^\ell}
		{\eta^{\ell+1}+\eta^\ell}\bigg|^p
		\Bigg]
	= \E(J^p) \le O(1) 
	\bigg( \f{k^2 I(\ell)}{2^{9k/20}} \bigg)^p\,,
	\]
which verifies the induction \eqref{e:stability.inductive} and proves the result.
\end{proof}
\end{lem}

We will apply Lemma~\ref{l:stability} below to obtain our final bound, Proposition~\ref{p:one.stable}, on $1$-stability in the $\uPGW$ tree. Before doing so, however, we note that Proposition~\ref{p:fp} is essentially an immediate consequence of Lemma~\ref{l:stability}:

\begin{proof}[Proof of Proposition~\ref{p:fp}]
Take the random sequence $(\bmeta^\ell)_{\ell\ge0}$ as in \eqref{e:coupling.of.all.eta.ell} or \eqref{e:coupling.eta.with.eps.param}, with $\EPS=0$. Then, as noted in the discussion around
\eqref{e:coupling.of.all.eta.ell}, the marginal law of each $\eta^\ell$ is precisely $\mu^\ell \equiv \Rec^\ell\mu^0$, the same as the $\mu^\ell$ appearing in the statement of this proposition. It follows from Lemma~\ref{l:stability} that for $p=2\lceil 2^{k/10} \rceil$ we have
	\[
	\sum_{\ell\ge0}
	\E\bigg[ \Big|\eta^{\ell+1}-\eta^\ell\Big|^p\bigg]^{1/p}
	\le 2\sum_{\ell\ge0} 
	\E\Bigg[
	\bigg|\f{\eta^{\ell+1}-\eta^\ell}
		{\eta^{\ell+1}+\eta^\ell}\bigg|^p
		\Bigg]^{1/p} < \infty\,.
	\]
Thus $(\eta^\ell)_{\ell\ge0}$ is a Cauchy sequence in $L^p$, hence it must converge in $L^p$ to a limiting random variable $\eta$ as $\ell\to\infty$. It follows that $\mu^\ell$ converges weakly to a limiting probability measure $\mu$ as $\ell\to\infty$. Since $\mu^\ell = \Rec\mu^{\ell-1}$ and the mapping $\Rec$ is continuous with respect to the weak topology on the space of distributions, we conclude $\mu = \Rec\mu$.
\end{proof}

The remainder of this subsection is devoted to the proof of Proposition~\ref{p:one.stable}. The following lemma records the easy observation that the canonical messages can only become ``more free'' when the neighborhood is enlarged:

\begin{lem}\label{l:support.of.pi.measures}
Let $\tree_{va}$ be any variable-to-clause tree in the sense of Definition~\ref{d:directed.trees}. Similarly as in 
\eqref{e:coupling.of.all.eta.ell} or \eqref{e:coupling.eta.with.eps.param}, let
	\[
	(\bmeta^\ell)_{\ell\ge0}
	\equiv \bigg(\FF_\ell(
	\tree_{va} )\bigg)_{\ell\ge0}\,.
	\]
It holds for any $\tree_{va}$ and any $\ell\ge0$ that if
$\bmeta^\ell(\plus)=0$ then $\bmeta^{\ell+1}(\plus)=0$ also;
likewise if $\bmeta^\ell(\minus)=0$ then $\bmeta^{\ell+1}(\minus)=0$ also.
Moreover $\bmeta^\ell(\free)\in(0,1]$ for all $\ell\ge1$. As a consequence we always have
	\beq\label{e:compare.support.of.pis}
	\set{\grn}\subseteq \supp\starpi\subseteq\supp\SQpi\,,\eeq
for all $r\ge2$,
for $\starpi$ and $\SQpi$ as in
Definition~\ref{d:canonical}.

\begin{proof}Thanks to the symmetry between $\plus$ and $\minus$, it suffices to prove the first assertion for the quantities $\eta^\ell \equiv \bmeta^\ell(\minus)$. That is to say, we shall argue that if
$\eta^\ell=0$ then $\eta^{\ell+1}=0$ also, for all $\ell\ge0$. Note that $\eta^0=1/2\ne0$ by assumption, so the statement holds trivially for $\ell=0$. Suppose inductively that it holds up to $\ell-1$; we then compare $\eta^\ell$ with $\eta^{\ell+1}$. Similarly as in Definitions~\ref{d:distributional.clause.recursion}--\ref{d:full.dist.recurs}
and the proof of Lemma~\ref{l:stability}, we can express
	\[
	\eta^{\ell+1} = \f{\Pi^\plus(1-\Pi^\minus)}
		{\Pi^\plus + \Pi^\minus - \Pi^\plus \Pi^\minus }\,,\quad
	\Pi^\PM\equiv \prod_{b\in\pd v(\PM a)} \bigg(
		1 - \prod_{u\in\pd b\setminus v} \eta^\ell(\tree_{ub})
		\bigg)
	\]
and similarly
	\[
	\eta^\ell = \f{\pi^\plus(1-\pi^\minus)}
		{\pi^\plus + \pi^\minus - \pi^\plus \pi^\minus }\,,\quad
	\pi^\PM\equiv \prod_{b\in\pd v(\PM a)} \bigg(
		1 - \prod_{u\in\pd b\setminus v} \eta^{\ell-1}(\tree_{ub})\,.
		\bigg)
	\]
Now suppose $\eta^\ell=0$. Since $\pi^\PM\in(0,1]$, it must be that $\pi^\minus=1$. This can occur in one of two possible ways:
\begin{enumerate}[--]
\item The first possibility is that $\pd v(\minus a)$. In this case $\Pi^\minus=1$ also, and so $\eta^{\ell+1}=0$. 
\item The only other possibility is that $\eta^{\ell-1}(\tree_{ub})=0$ for all $u\in\pd b\setminus v$, for all $b\in\pd v(\minus a)$. In this case, it follows from the inductive hypothesis that $\eta^\ell(\tree_{ub})=0$ for all $u\in\pd b\setminus v$, for all $b\in\pd v(\minus a)$. Therefore $\Pi^\minus=1$ also, and so again we conclude $\eta^{\ell+1}=0$. 
\end{enumerate}
This verifies the induction and thus proves the first assertion. It is easy to see from the form of the recursion \eqref{e:eta.in.terms.of.signed.Pis} that
$\bmeta^\ell(\free)\in(0,1]$ for all $\ell\ge1$, since $\Pi^\PM\in(0,1]$.
For the next assertion \eqref{e:compare.support.of.pis}, recall from the discussion around \eqref{e:defn.canonical.eta} and \eqref{e:defn.canonical.edge.marginal} that $\starpi$ is proportional to 
the product of $\hqstar_{av}$ and
$\dqstar_{va}$, while $\SQpi$ is proportional to
the product of $\hqstar_{av}$ and
$\SQdq_{va}$. The only difference is that
$\dqstar_{va}$ corresponds to $\bmeta^r$
while $\SQdq_{va}$ corresponds to $\bmeta^{r-1}$
(via \eqref{e:color.recursions.eta}). It then follows from the prior assertions of this lemma that we have
	\[
	\set{\free}\subseteq \supp\bmeta^r
	\subseteq \supp\bmeta^{r-1}
	\]
for all $r\ge2$, and \eqref{e:compare.support.of.pis} follows straightforwardly by substituting into
\eqref{e:color.recursions.eta}.
\end{proof}
\end{lem}

\begin{lem} \label{l:root.is.stable.EPSILON}
For the measure $\uPGW_\EPS$ of Definition~\ref{d:PGW.with.clauses.one.less},
we have
	\[
	\uPGW_\EPS\Big(
		\textup{$\vrt$ not stable}\Big)
	\le \f1{\exp(2^{k/12} R)}\,,
	\]
where ``stable'' means $0$-stable in the sense of Definition~\ref{d:j.stable}. This bound holds for all $0\le\EPS\le1$.

\begin{proof}
For any tree $\tree$ and any edge $(au)$ of $\tree$,
we can consider the variable-to-clause
 tree $\tree_{ua}\subseteq\tree$, and define (cf.\ \eqref{e:coupling.of.all.eta.ell}~and~\eqref{e:coupling.eta.with.eps.param})
the measures $\bmeta^\ell_{ua}\equiv \FF_\ell(\tree_{ua})$.
We define two functions on variable-to-clause trees,
	\begin{align*}
	f^\plus(\tree_{ua})
	&\equiv
	\mathbf{1}\Bigg\{
	\max\bigg\{
	\bmeta^r_{ua}(\plus),
	\bmeta^{r-1}_{ua}(\plus)
	\bigg\} \ge 
	1-\f1{k^r}\Bigg\}\,,\\
	g^\plus(\tree_{ua})
	&\equiv \mathbf{1} \Bigg\{\Bigg|
		\f{\bmeta^r_{ua}(\plus)
			-\bmeta^{r-1}_{ua}(\plus)}
		{\bmeta^r_{ua}(\plus)+
		\bmeta^{r-1}_{ua}(\plus)}\Bigg|
	\ge \f1{2^{kr/8}}
	\Bigg\}\,,
	\end{align*}
and likewise 
$f^\minus(\tree_{ua})$
and $g^\minus(\tree_{ua})$.
If $u$ and $w$ are neighboring variables on $\tree$,
we let $a(uw)\equiv a(wu)$ denote the unique clause that they share. For the rest of the proof, we let $\tree$ be a sample from $\uPGW_\EPS$, rooted at $v\equiv\vrt$.
Since $\uPGW_\EPS$ is unimodular, we can apply \eqref{e:unimodular} to evaluate
	\beq\label{e:unimodular.applied}
	\int\bigg[ \sum_{u\in N(v)}
	f^\minus(\tree_{v a(v u)})\bigg]
	\,d\uPGW_\EPS(\tree)
	=\int \bigg[\sum_{u\in N(v)}
	f^\minus(\tree_{u a(uv)})\bigg]
	\,d\uPGW_\EPS(\tree)\,.
	\eeq
Note that, given $B_1(v)$,
the $f^\minus(\tree_{v a(v u)})$ are not conditionally independent over $u\in N(v)$,
but the $f^\minus(\tree_{u a(uv)})$ are.
The above is therefore upper bounded by
	\beq\label{e:not.too.forcing.pgw.bound}
	k^2\alpha
	\int f^\minus(\tree_{va})
	\,d\PGW_\EPS(\tree_{va})
	\le
	\f{k^2\alpha}{\exp(\Omega(r 2^{k/4} ))}\,,\eeq
where the last bound is
by the upper tail bound \eqref{i:mu.ell.upper.tail.bound} from Lemma~\ref{l:mu.ell.tail.bounds}.
By symmetry, 
the bound \eqref{e:not.too.forcing.pgw.bound}
also holds with $f^\plus$
in place of $f^\minus$.
With $g^\minus$ in place of $f^\minus$, the identity \eqref{e:unimodular.applied} also holds, and conditional independence gives the upper bound
	\[
	k^2\alpha
	\int g^\minus(\tree_{va})
	\,d\PGW_\EPS(\tree_{va})
	\le
	k^2\alpha
	\,\P_\EPS\Bigg(
	 \Bigg|\f{\eta^r-\eta^{r-1}}
		{\eta^r+\eta^{r-1}}\Bigg|
	\ge \f1{2^{kr/8}}
	\Bigg)\,,
	\]
where $\P_\EPS$ is the measure from Lemma~\ref{l:stability}. From the bound of 
Lemma~\ref{l:stability}, the last expression is
	\beq\label{e:stability.pgw.bound}
	\le
	k^2\alpha(2^{kr/8})^p
	\E_\EPS\Bigg[\Bigg(
	 \Bigg|\f{\eta^r-\eta^{r-1}}
		{\eta^r+\eta^{r-1}}\Bigg|^p\Bigg]
	\le
	\f{k^2\alpha 2^{kpr/8} }{2^{kpr/7}}
	\le
	\f{4^k}{2^{kpr/56}}
	\le \f1{\exp(k r 2^{k/11} )}\,.
	\eeq
Let $A=f^\plus+f^\minus+g^\plus+g^\minus$, and define
the ($B_r(v;\tree)$-measurable) event
	\[
	\bm{E}\equiv
	\bigcap_{a\in\pd v}
	\bigg\{A(\tree_{ua})=0
	\textup{ for all }u\in\pd a
	\bigg\}\,.\]
We then find by Markov's inequality together with \eqref{e:unimodular.applied}, \eqref{e:not.too.forcing.pgw.bound}, and \eqref{e:stability.pgw.bound} that
	\[
	\uPGW_\EPS(\bm{E}^c)
	\le
	\int\bigg[\sum_{u\in\pd v}\bigg\{
	\f{A(\tree_{v a(v u)})}{k-2}
	+A(\tree_{u a(u v)})\bigg\}\bigg]
	\,d\uPGW_\EPS(\tree)
	\le \f1{\exp(r 2^{k/11})}\,.
	\]
We will prove that the root $v$ is stable on event $\bm{E}$. 

Recall from Definition~\ref{d:j.stable} that $v$ is stable if all its incident edges $e\in\delta v$ are both message-stable and marginal-stable (Definition~\ref{d:stable}). It is clear that $\bm{E}$ implies that every edge incident to $v$ satisfies the message-stability condition \eqref{e:message.stability}. (The first part of \eqref{e:message.stability} holds because $\bm{E}$ implies $f^\PM(\tree_{ua})=0$ for all $a\in \pd v$ and $u\in\pd a$. The second part of \eqref{e:message.stability} holds because $\bm{E}$ implies $g^\PM(\tree_{ua})=0$ for all $a\in \pd v$ and $u\in\pd a$.) Therefore it remains only to check the marginal-stability conditions \eqref{e:stable.yellow}--\eqref{e:stable.positivity}.

For the rest of the proof, for all $a\in\pd v$ and $u\in\pd a$, we will denote
$H_{ua}\equiv \bmeta^r_{ua}(\minus)$,
$h_{ua}\equiv \bmeta^{r-1}_{ua}(\minus)$,
$P_{ua}\equiv \bmeta^r_{ua}(\plus)$, and
$p_{ua}\equiv \bmeta^{r-1}_{ua}(\plus)$.
For $a\in\pd v$ we also let
	\[
	\hat{u}_{av}
	\equiv \prod_{u\in\pd a\setminus v} h_{ua}\,.
	\]
With this notation, and using the correspondence
\eqref{e:color.recursions.eta},
we have for all $a\in\pd v$ that
	\[
	\bigg(\starpi_{av}(\red),
	\starpi_{av}(\yel),
	\starpi_{av}(\blu),
	\starpi_{av}(\cya)
	\bigg)
	= \bigg(
		\f{(1-H_{va})\hat{u}_{av}}
		{1-\hat{u}_{av} H_{va}},
		\f{H_{va}(1-\hat{u}_{av})}
		{1-\hat{u}_{av} H_{va}},
		\f{
		P_{va}(1-\hat{u}_{av})}
		{1-\hat{u}_{av} H_{va}},
		\f{
		(1-H_{va})(1-\hat{u}_{av})
		}{1-\hat{u}_{av} H_{va}}
	\bigg)\,.
	\]
We can obtain $\SQpi_{av}$ from the same expressions, only $h_{va}$ and $p_{va}$ in place of $H_{va}$ and $P_{va}$ (keeping $\hat{u}_{av}$ the same). By Lemma~\ref{l:support.of.pi.measures},
if $h_{va}=0$ then $H_{va}=0$,
and if $p_{va}=0$ then $P_{va}=0$.
On the event $\bm{E}$, if $(H,h,\hat{u})=(H_{va},h_{va},\hat{u}_{av})$ for any $a\in\pd v$, then either 
$h=H=0$, or $h>0$ and
	\begin{align*}
	\bigg|\f{H}{h}-1\bigg|= \f{|H-h|}{h}
	\le 2\bigg|\f{H-h}{H+h}\bigg|
		\bigg/
		\Bigg( 1-\bigg|\f{H-h}{H+h}\bigg|\Bigg)
	\le \f{O(1)}{2^{kr/8}}
	&\le \f{1}{2^{kr/9} k^{10}} \,,\\
	\f{|(1-\hat{u} H)-(1-\hat{u} h)|}
		{1-\hat{u}h}
	\le
	\f{|H-h|}{1-h}
	\le k^r|H-h|
	\le 2 k^r
	\bigg|\f{H-h}{H+h}\bigg|
	\le \f{2 k^r}{2^{kr/8}}
	&\le \f{1}{2^{kr/9} k^{10}} \,.
	\end{align*}
The analogous bounds hold with $(P,p)=(P_{va},p_{va})$ in place of $(H,h)$. It follows that for all $\sigma\in\set{\red,\yel,\blu}$, we have either $\SQpi_{av}(\sigma)=\starpi_{av}(\sigma)=0$,
or $\SQpi_{av}(\sigma)>0$ and
	\beq\label{e:yrc.diff.div.y}
	\bigg|\f{\starpi_{av}(\sigma)-\SQpi_{av}(\sigma)}
		{\SQpi_{av}(\sigma)}\bigg|
	\le \f1{2^{kr/9}}
	\eeq
For $\sigma=\cya$,
we know that 
$\SQpi_{av}(\cya)\ge \SQpi_{av}(\grn)$ is always positive, and the estimate \eqref{e:yrc.diff.div.y} holds on all of $\bm{E}$. This implies the last stability condition~\eqref{e:stable.positivity}.
It remains to verify the other three conditions
\eqref{e:stable.yellow},
\eqref{e:stable.cyan}, and \eqref{e:stable.red}.
This essentially amounts to a quantitative version of the argument of Lemma~\ref{l:clause.based.marginals.cohere}.
Using \eqref{e:clause.tuple.measure.weighted}
and the correspondence \eqref{e:color.recursions.eta}, we have
	\[
	\COHER_{au}(\SQpi)
	= \f{h_{ua}}{\hat{z}_a}
	\bigg[1-\prod_{w\in\pd a\setminus u} h_{wa}
	-\sum_{w\in\pd a\setminus u} (1-h_{wa})
		\prod_{z\in\pd a \setminus \set{u,w}} h_{za}
	\bigg]\,,
	\]
where $\hat{z}_a$ is the clause normalization, given explicitly by
	\[
	\hat{z}_a
	= 1-\prod_{w\in\pd a} h_{wa} \le 1\,.
	\]
If $w,z$ are any two distinct variables in $\pd a\setminus u$, on the event $\bm{E}$ we have the crude lower bound
	\[
	\COHER_{au}(\SQpi)
	\ge \f{h_{ua}}{\hat{z}_a} 
		(1-h_{wa})(1-h_{za})
	\ge \f{h_{ua}}{\hat{z}_a k^{2r}}\,.
	\]
Similarly, if $u,w,z$ are any three distinct variables in $\pd a$, then 
 on the event $\bm{E}$ we have 
	\[
	\COHER_a(\SQpi)
	\ge\f1{\hat{z}_a}
	(1-\eta_{ua})
	(1-\eta_{wa})
	(1-\eta_{za})
	\ge \f1{\hat{z}_a k^{3r}}
	\ge \f1{k^{3r}}\,.
	\]
On the other hand, we have
	\beq\label{e:yellow.vs.coher.edge}
	\SQpi_{au}(\yel)
	=\f{h_{ua}}{\hat{z}_a}
	\bigg[1-\prod_{w\in\pd a\setminus u} h_{wa}\bigg]
	\le\f{h_{ua}}{\hat{z}_a}
	\le k^{2r} \COHER_{au}(\SQpi)\,.
	\eeq
Similarly, for all $w\in\pd a\setminus u$ we have
	\beq\label{e:red.vs.coher.edge}
	\SQpi_{aw}(\red)
	=\f{(1-h_{wa})}{\hat{z}_a}
	\prod_{z\in\pd a\setminus w} h_{za}
	\le \f{h_{ua}}{\hat{z}_a}
	\le k^{2r} \COHER_{au}(\SQpi)\,.
	\eeq
For all $u\in\pd a$, we have trivially
	\beq\label{e:coher.a.is.large}
	\bigg\{
	\SQpi_{au}(\cya)+\SQpi_{au}(\red)
	\bigg\}
	\le 1 \le k^{3r} \COHER_a(\SQpi)\,.
	\eeq
The bounds \eqref{e:stable.yellow},
\eqref{e:stable.cyan}, and \eqref{e:stable.red}
follow by combining
 \eqref{e:yrc.diff.div.y} with
\eqref{e:yellow.vs.coher.edge}, \eqref{e:red.vs.coher.edge}, \eqref{e:coher.a.is.large}.
\end{proof}
\end{lem}

We now apply Lemma~\ref{l:root.is.stable.EPSILON} to conclude the proof of the main result of this subsection:

\begin{proof}[Proof of Proposition~\ref{p:one.stable}]
We can sample from $\uPGW_\EPS$ in the following way: first sample $\tree\sim\uPGW$. Declare each clause of $\tree$ to be ``open'' with probability $\EPS$, independently over all clauses. For each open clause $a$, choose one of its child variables $u\in\pd a$ uniformly at random, and declare $u$ to be ``open'' as well. Let $V'$ be the (random) set of all open variables in $B_R(\vrt;\tree)$, and let $\tree'$ be the connected component of $\tree\setminus V'$ that contains $\vrt$ --- as a shorthand we will write $\tree'\equiv \tree'(V')$. Then $\tree'$ is a sample from $\uPGW_\EPS$. Let $\gamma_\EPS(\cdot\,|\,\tree)$ denote the law of $V'$ given $\tree$. 
Let $\Omega'$ denote the subspace of $\tree$ for which 
$|B_R(\vrt;\tree)| \le M'$, where $M'$ will be chosen below. Then
	\begin{align*}
	P(\EPS)&\equiv \uPGW_\EPS\Big(\textup{$\vrt$
		is not stable}\Big)
	\ge
	\int_{\Omega'} \int
	\mathbf{1}\Big\{
	\textup{$\vrt$ not stable in $\tree'(V')$}
	\Big\}
	\,d\gamma_\EPS(V'\,|\,\tree)
	\,d\uPGW(\tree)\\
	&\ge\int_{\Omega'}
	\sum_{u \in B_R(\vrt;\tree) \setminus \vrt}
	\mathbf{1}\Big\{
	\textup{$\vrt$ not stable in $\tree'(\set{u})$}
	\Big\}
	\int
	\mathbf{1}\Big\{V' = \set{u}
	\Big\}\,d\gamma_\EPS(V'\,|\,\tree)
	\,d\uPGW(\tree)\\
	&\ge 
	\f{\EPS(1-\EPS)^{M'}}{M'}
	\int_{\Omega'}
	\sum_{u \in B_R(\vrt;\tree) \setminus \vrt}
	\mathbf{1}\Big\{
	\textup{$\vrt$ not stable in $\tree'(\set{u})$}
	\Big\}\,d\uPGW(\tree)\,,
	\end{align*}
where, for any $\tree\in\Omega'$ and any $u\in B_R(\vrt;\tree)\setminus\set{\vrt}$, the quantity $\EPS(1-\EPS)^{M'}/M'$ a crude lower bound on the chance that $V'=\set{u}$. If we rearrange the above and consider the contribution outside $\Omega'$, we obtain
	\begin{align*}
	P' &\equiv \int
	\sum_{u \in B_R(\vrt;\tree) \setminus \vrt}
	\mathbf{1}\Big\{
	\textup{$\vrt$ not stable in $\tree'(\set{u})$}
	\Big\}\,d\uPGW(\tree)\\
	&\le\f{P(\EPS) M'}{\EPS(1-\EPS)^{M'}}
	+ \E\bigg[\Big|B_R(\vrt;\tree)\Big| 
		; \Big|B_R(\vrt;\tree)\Big| \ge M'\bigg]
	\le
	\f1{\exp(\Omega(2^{k/12}R))}
	\,,
	\end{align*}
where the last bound follows by 
taking $M'=1/\EPS = \exp(k^2R)$, and applying
Lemma~\ref{l:martingale.bound.PGW} with Lemma~\ref{l:root.is.stable.EPSILON}.
Combining the above with the $\EPS=0$ case of 
Lemma~\ref{l:root.is.stable.EPSILON} gives
	\[
	\uPGW\Big(\textup{$\vrt$ not $1$-stable}\Big)
	\le
	\uPGW\Big(\textup{$\vrt$ not stable}\Big)
	+P'
	\le \f{O(1)}{\exp(\Omega(2^{k/12}R))}\,,
	\]
and the claimed bound follows.
\end{proof}

\subsection{Threshold upper bound}\label{ss:threshold.ubd}

In this subsection we complete the proof of Proposition~\ref{p:ubd}. Recall that in \S\ref{ss:pgw.stability} we proved Proposition~\ref{p:fp}, saying that the sequence of measures $\mu^\ell$ converges weakly to a 
limit $\mu\equiv\mu^{\infty,\alpha}$ which satisfies the distributional fixed point equation
$\Rec\mu=\mu$. Given this result, the $\onersb$ free energy can be written as
 	\beq\label{e:phi.onersb}
	\Phi(\alpha)
	\equiv \sum_{\vec d} \POpm(\vec d)
	\int
	\log\f{\Pi^\plus(\vec d,\vec\eta)
		+\Pi^\minus(\vec d,\vec\eta)-
		\Pi^\plus(\vec d,\vec\eta)\Pi^\minus(\vec d,\vec\eta) }
		{(1- \prod_{j=1}^k \eta_j)^{(k-1)\alpha} }
	\, d\mu^{\otimes}\Big(
	(\eta_j)_{j\ge1},\ueta
	\Big)\,.
	\eeq
(This is the same as \eqref{e:phi.alpha}.) We defer to Section~\ref{s:monotonicity} the proof of Proposition~\ref{p:phi}, which guarantees that $\Phi$ is decreasing in $\alpha$ so that the conjectured threshold $\arsb$ is well-defined. In the current subsection, we give the proof of Proposition~\ref{p:ubd} assuming that Proposition~\ref{p:phi} holds. The proof is an application of interpolation bounds \cite{FrLe:03,MR2095932} on the free energy of positive-temperature dilute spin glasses. We believe this argument was generally known, especially among the physics community; we include the proof of Proposition~\ref{p:ubd} only for the sake of completeness.

The \bemph{positive-temperature $\ksat$ model} can be formally defined as follows.
First let $(\lit_{aj})_{a,j\ge0}$ be an array of i.i.d.\ symmetric random signs, $\lit_{aj}\in\set{\PM}$ with $\P(\lit_{aj}=\plus)
	= \P(\lit_{aj}=\minus)=1/2$.
We then use these to define
a vector $\vec\theta\equiv(\theta_a)_{a\ge0}$ of
 i.i.d.\ random functions
	\[
	\theta_a(x_1,\ldots,x_k)
	\equiv \mathbf{1}\Big\{\lit_{aj}x_j =\minus
		\text{ for all }
		1\le j\le k\Big\}\,.\]
Let $M$ be a Poisson random variable with mean $n\alpha$. Define the \textbf{random $\ksat$ Hamiltonian}
$H_n:\set{\PM}^n\to [0,M]$,
	\[
	H_n(\ux)
	= \sum_{a=1}^M \theta_a( \ux_{\pd a} )\,,\]
where $\pd a$ are chosen independently and uniformly at random from $[n]^k$. The \bemph{random $\ksat$ free energy} at inverse temperature $\beta$ is then given by
	\[
	F_n(\beta)
	= \f1n \E_n \log
		\sum_{ \ux \in\set{\PM}^n }
		\exp\{ -\beta H_n(\ux) \}
	\]
with $\E_n$ denoting expectation over the random Hamiltonian $H_n$. The next bound is from \cite[Theorem~3]{MR2095932}, edited only slightly to fit our notation. (In the remainder of this subsection, we will no longer refer to the random scalar $\eta$ from \S\ref{ss:dist.rec}, but will instead use $\eta$ to denote a certain random measure which is closely related.)

\begin{thm}[{\cite[Theorem~3]{MR2095932}\footnote{Although the theorem in \cite{MR2095932} is stated for even $k$, the same proof applies equally to odd $k$, as noted for example in \cite[Ch.~6]{MR2731561}.}}]\label{t:free.energy} Let $\mathscr{M}_1$ denote the space of probability measures on $\mathbb{R}$, and $\mathscr{M}_2$ the space of probability measures on $\mathscr{M}_1$. For $\zeta\in\mathscr{M}_2$, let $\ueta\equiv(\eta_{a,j})_{a,j\ge0}$ be an array of i.i.d.\ samples from $\zeta$. Conditioned on $\ueta$, let $\vec\rho\equiv(\rho_{a,j})_{a,j\ge0}$ where each $\rho_{a,j}$ is a conditionally independent sample from $\eta_{a,j}$. For $x\in\set{\PM}$, define
	\[\bm{u}_a(x)
	= \sum_{\ux \in\set{\PM}^k}
		\Ind{x_k=x}\exp\{ -\beta\theta_a(\ux)\}
		\prod_{j=1}^{k-1} 
	\f{ \exp\{\rho_{a,j} x_j\}}
		{ 2 \ch \rho_{a,j} }\,,\]
with $\ch$ the hyperbolic cosine. Define also
	\[\bm{u}_a
	= \sum_{\ux \in\set{\PM}^k}
		\exp\{ -\beta\theta_a(\ux)\}
		\prod_{j=1}^k \f{\exp\{\rho_{a,j} x_j\}}
			{2\ch \rho_{a,j}}\,.\]
Note $\bm{u}_a(x)$ and $\bm{u}_a$ are random variables. Then, for any $0<m<1$ and for any $\zeta\in\mathscr{M}_2$,
	\[
	F_n(\beta)
	\le
	\Phi_{\onersb}(\beta,\zeta,m)
	\equiv \f1m \E\log
		\E'\bigg[
		\Big(
		\sum_{x\in\set{\PM}}
		\prod_{a=1}^d \bm{u}_a(x)
		\Big)^m
		\bigg]
	- \f{(k-1)\alpha}{m}
		\E\log\E'[(\bm{u}_0)^m]
	\]
where $\E'$ is expectation over $\vec\rho$ conditioned on $d,\vec\theta,\vec\eta$; and $\E$ is the overall expectation over $d,\vec\theta,\vec\eta,\vec\rho$.
\end{thm}

The threshold upper bound is a straightforward consequence:

\begin{proof}[Proof of Proposition~\ref{p:ubd}]
We will deduce the bound from Theorem~\ref{t:free.energy} by taking a particular choice of $\zeta,m$ which is suggested by the survey propagation heuristic. Let $\mu\in\PINT$ be the fixed point given by Proposition~\ref{p:fp}. Let $\bmu=\REC\mu$ for $\REC\equiv\REC_0$ as given by Definition~\ref{d:full.dist.recurs}. Now fix $\beta>0$, and let
	\beq\label{e:interpolation.rho}
	\Big(\rho_\plus,\rho_\minus,\rho_\free\Big)
		\equiv
		\Big(\beta,-\beta,0\Big)\,.
	\eeq
Let $\eta$ be defined as the law of $\rho_y$ where $y\in\set{\plus,\minus,\free}$ is distributed according to $\bmeta$, and $\bmeta$ is distributed according to $\bmu$: formally, for any Borel set $B\subseteq\mathbb{R}$ (where it suffices to consider $B\subseteq\set{\beta,-\beta,0}$),
	\[
	\eta(B)
	= \int \bmeta(y) \mathbf{1}\Big\{\rho_y\in B\Big\}
	 \,d\bmu(\bmeta)\,.
	\]
Thus $\eta\in\mathscr{M}_1$ is an $\bmeta$-measurable random measure supported on $\set{\rho_\plus,\rho_\minus,\rho_\free}=\set{\beta,-\beta,0}\subseteq\mathbb{R}$. Let $\zeta$ denote the law of $\eta$, so $\zeta\in\mathscr{M}_2$ and the randomness in $\zeta$ is the randomness of $\bmeta$. For $x\in\set{\minus,\plus}$ and $\uy_a\in\set{\plus,\minus,\free}^{k-1}$, let
	\[
	\bm{u}_a(x|\uy_a)
	\equiv \sum_{\ux \in\set{\PM}^k}
		\f{\Ind{x_k=x}}
			{\exp(\beta\theta_a(\ux))}
		\prod_{j=1}^{k-1} 
		r(x_j | y_{a,j})\,,\quad
		r(x|y) \equiv 
		\f{\exp\{\rho_y x\}}{2\ch\rho_y}\,.\]
Note it follows from the definition 
\eqref{e:interpolation.rho} that 
	\[
	r(\plus|\plus)
	=r(\minus|\minus)
	= \f{e^\beta}{2\ch\beta}
	= 1-r(\minus|\plus)
	= 1-r(\plus|\minus)\,.
	\]
It follows by combining with the definition of $\theta_a$ that
	\begin{align*}
	\bm{u}_a\Big( \minus\lit_{a,k} \,\Big|\,
		(\minus\lit_{a,j})_{1\le j\le k-1}
		\Big)
	&= \f1{e^\beta}
	\prod_{j=1}^{k-1}
	r( \minus\lit_{a,j}|\minus\lit_{a,j} )
	+ \bigg\{
	1-	\prod_{j=1}^{k-1}
	r( \minus\lit_{a,j}|\minus\lit_{a,j} )
	\bigg\} \\
	&=\f1{e^\beta} \bigg(
		\f{e^\beta}{2\ch\beta}\bigg)^{k-1}
	+\bigg\{ 1- \bigg(
		\f{e^\beta}{2\ch\beta}\bigg)^{k-1}
		\bigg\}
	\le \f1{e^{\beta/2}}\,,
	\end{align*}
where the last bound holds assuming $k$ is large enough and $\beta\ge k$. Thus, with $\P'$ denoting the law of $\vec\rho$
conditioned on $d,\vec\theta,\vec\eta$, for $x\in\set{\PM}$ we have
	\[
	\P'\bigg(
	\underbrace{\prod_{a=1}^d \bm{u}_a(x)
		> \f1{e^{\beta/2}}}_{\textup{event }F_x}\bigg)
	\le \Pi^x
	\equiv\prod_{a : \lit_{a,k}=x}
		\bigg( 1-\prod_{j=1}^{k-1}
		\bmeta_{a,j}
			( \minus\lit_{a,j} )
		\bigg)\,.
	\]
Since $F_x$
can be decided by looking only at the clauses $\pd v(x)$, we have
	$\P'( F_\plus \cup F_\minus )
	\le 1 - (1-\Pi^\plus)(1-\Pi^\minus)$, so
	\[\E'\bigg[
		\Big(
		\sum_{x\in\set{\PM}}
		\prod_{a=1}^d \bm{u}_a(x)
		\Big)^m
		\bigg]
	\le 2^m \P'\Big(F_\plus \cup F_\minus\Big)
		+ \bigg( \f{2}{e^{\beta/2}}\bigg)^m
	\le 2^m \bigg( \Pi^\plus+\Pi^\minus
		-\Pi^\plus\Pi^\minus 
		+ \f1{e^{m\beta/2}}
		\bigg)\,.\]
Similar considerations give 
	\[
	\E'\Big[ (\bm{u}_a)^m \Big]
	\ge
	\f1{2^m}
	\P'\bigg(\bm{u}_a\ge\f12\bigg)
	\ge \f1{2^m}\bigg(
	1 - \prod_{j=1}^k \eta_{a,j}
		(\minus\lit_{a,j})\bigg)\,.
	\]
Combining the above bounds gives
	\[
	m \Phi_{\onersb}(\beta,\zeta,m)
	\le m 4^k + 
	\E\log \f{\Pi^\plus+\Pi^\minus
		-\Pi^\plus\Pi^\minus 
		+ e^{-m\beta/2}}
		{ [1-\prod_{j=1}^{k-1} \eta_j]^{(k-1)\alpha} }.
	\]
Taking $m=1/\beta^{1/2}$, the right-hand side converges to the function $\Phi(\alpha)$ from \eqref{e:phi.onersb} in the limit $\beta\to\infty$. Now assume $\alpha>\arsb$, so that $\Phi(\alpha)<0$ (of course, this uses Proposition~\ref{p:phi}). Then, for sufficiently large $\beta$, we will have
	\beq\label{e:neg.free.energy}
	F_n(\beta)
	\le
	\Phi_{\onersb}\bigg(\beta,\zeta,\f1{\beta^{1/2}}\bigg)
	\le 4^k + 
	\f{\beta^{1/2} \Phi(\alpha)}{2}
	\le \f{\beta^{1/2} \Phi(\alpha)}{4} 
	< 0.
	\eeq
To conclude, recall that the log-partition function $\log Z_n(\beta)$ is well concentrated about its expected value $nF_n(\beta)$ (take the Doob martingale of $\log Z_n(\beta)$ with respect to the $\ksat$ clause-revealing filtration, and apply the Azuma--Hoeffding inequality). Thus \eqref{e:neg.free.energy} implies that with high probability $\log Z_n(\beta)$ is negative, i.e., the \textsc{sat} instance is unsatisfiable.
\end{proof}

\subsection{First moment of judicious colorings}
\label{ss:judicious.first.mmt}

In this subsection we complete the proof of Proposition~\ref{p:first.moment.exponent} --- recall the statement of the proposition is that
	\[
	\E_{\DD}\ZZ
	\ge
	\exp\bigg\{ n\Big[
		\Phi(\alpha)-o_R(1)\Big]\bigg\}
	\]
with high probability over $\DD$,
where $\Phi(\alpha)$ is the
$\onersb$ free energy
\eqref{e:phi.alpha}. We pick up from the discussion of \S\ref{ss:first.moment.preliminaries} where several preliminary calculations were done. In Corollary~\ref{c:first.moment.exponent} we saw that
	\[
	\E_{\DD}\ZZ
	= \exp\bigg\{ 
	n\bm{\Psi}_{\DD}(\omstar) - o(n)
	\bigg\}
	\]
where $\bm{\Psi}_{\DD}(\omstar)$ is defined by the constrained optimization problem
\eqref{e:nu.opt}. Below
Corollary~\ref{c:first.moment.exponent},
we discussed that the natural guess for the limiting value of $\bm{\Psi}_{\DD}(\omstar)$
(as $n\to\infty$ and $R\to\infty$)
is $\Phi^\textup{col}(\alpha)$,
the Bethe free energy of the coloring model
(Definition~\ref{d:Phi.col}). We first verify that 
this coincides precisely with the $\onersb$
free energy $\Phi(\alpha)$ of the original model:

\begin{lem}\label{l:phi.equals.phi.col}
For $\alpha$ in the regime
\eqref{e:alpha.regime},
the $\onersb$ free energy of the $\ksat$ model
($\Phi(\alpha)$ from \eqref{e:phi.alpha})
coincides with the replica symmetric
(Bethe) free energy of the coloring model
($\Phi^\textup{col}(\alpha)$
from Definition~\ref{d:Phi.col}).

\begin{proof}We will show that the form of $\Phi^\textup{col}(\alpha)$ given by \eqref{e:simplified.phi.col.terms}~and~\eqref{e:simplified.phi.col} is equivalent to $\Phi(\alpha)$. As in Definition~\ref{d:bethe.laws.of.color.messages}, let $\bmeta_j$ be i.i.d.\ copies of $\bmeta$, and let $\bhu$ as in \eqref{e:second.defn.of.bhu}. Let $\dq_j\equiv \dq(\bmeta_j)$ via the correspondence \eqref{e:color.recursions.eta.again}. Then the $\bar{z}$ in \eqref{e:simplified.phi.col.terms} can be re-written as
	\[
	\bar{z}
	= \sum_\sigma 
		\dq_k(\sigma)
		\hq(\sigma)
	= \dq_k(\red) \hq(\red)
		+ [1-\dq_k(\red)]\hq(\yel)
	=
	 \f{1}
		{ [3-2\bm{\hat{u}} (\plus)]
			[2-\bmeta_k(\minus)] }
			\Bigg(1 -
			\prod_{j=1}^k \bmeta_j(\minus) \Bigg)\,.\]
Similarly, the 
$\bm{\hat{z}}$ in \eqref{e:simplified.phi.col.terms}
can be re-written as
	\begin{align*}
	\bm{\hat{z}}
	&=\sum_{\usi} \hat{\varphi}(\usi)
		\prod_{j=1}^k \dq_j(\sigma_j)
	= \prod_{j=1}^k [1-\dq_j(\red)]
	-\prod_{j=1}^k \dq_j(\yel)
	+ \sum_{j=1}^k [\dq_j(\red)-\dq_j(\cya)]
		\prod_{l\ne j} \dq_l(\yel)\\
	&=\prod_{j=1}^k [1-\dq_j(\red)]
		-\prod_{j=1}^k \dq_j(\yel)
	= \Bigg(1- \prod_{j=1}^k\bmeta_j(\minus) \bigg)
		\Bigg/
		\Bigg(
		\prod_{j=1}^k [2-\bmeta_j(\minus)]
		\Bigg)\,.
	\end{align*}
For the $\dbz$ in \eqref{e:simplified.phi.col.terms}, 
given $\ud=(d^\plus,d^\minus)$ let us abbreviate
	\[
	\usi =
	\bigg(
	\Big(\sigma^\plus_i\Big)_{1\le i\le d^\plus},
	\Big(\sigma^\minus_i\Big)_{1\le i\le d^\minus}
	\bigg)\,,
	\]
and let $\varphi(\usi;\vec d)$ be the indicator that
$\usi$ gives a valid coloring of a variable where the first $d^\plus$ incident edges are of $\plus$ sign
while the remaining $d^\minus$ 
incident edges are of $\minus$ sign. Then we can rewrite
	\begin{align*}
	\dbz
	&= \sum_{\usi}
		\varphi(\usi;\vec d)
		\prod_{i=1}^{d^\plus} \hq^\plus_i(\sigma^\plus_i)
		\prod_{i=1}^{d^\minus}
		\hq^\minus_i(\sigma^\minus_i)
	=\prod_{i=1}^{d^\plus}
		\hq^\plus_i(\grn)
		\prod_{i=1}^{d^\minus}
		\hq^\minus_i(\grn)\\ \nonumber
	&\qquad + \Bigg\{
		\prod_{i=1}^{d^\plus}
		\hq^\plus_i(\pur)
		-\prod_{i=1}^{d^\plus}
		\hq^\plus_i(\blu)\Bigg\}
		\prod_{i=1}^{d^\minus}
		\hq^\minus_i(\yel)
	+ \Bigg\{
		\prod_{i=1}^{d^\plus}
		\hq^\plus_i(\pur)
		-\prod_{i=1}^{d^\plus}
		\hq^\plus_i(\blu)\Bigg\}
		\prod_{i=1}^{d^\minus}
		\hq^\minus_i(\yel)\\
	&= \bigg(\Pi^\plus+\Pi^\minus
		-\Pi^\plus\Pi^\minus\bigg)
		\bigg/
		\Bigg\{
		\prod_{i=1}^{d^\plus}
		[3-2\bm{\hat{u}}^\plus_i(\plus)]
		\prod_{i=1}^{d^\minus}
		[3-2\bm{\hat{u}}^\minus_i(\plus)]
		\Bigg\}\,,\quad
		\bm{\hat{u}}^\PM_i(\plus)
		\equiv \prod_{j=1}^{k-1}
			\bmeta^\PM_{ij}(\minus).
	\end{align*}
Let us write simply $\E$ for expectation with respect to the law of the $\bmeta_i$. Substituting the above into \eqref{e:simplified.phi.col.terms} gives
	\begin{align*}
	\Phi^{\textup{col},\textup{v}}(\alpha)
	&=\E \Bigg[\log \bigg(\Pi^\plus+\Pi^\minus
		-\Pi^\plus\Pi^\minus\bigg)\Bigg]
	-k\alpha\,\E
		\Bigg[\log\bigg(3-2\bhu(\plus)\bigg)\Bigg]\,,\\
	\alpha\,\Phi^{\textup{col},\textup{cl}}(\alpha)
	&=
	\alpha\,\E\Bigg[
	\log\bigg(1-\prod_{j=1}^k
	\bmeta_j(\minus)\bigg)\Bigg]
	-k\alpha\,\E
	\bigg[\log\bigg(2-\bmeta_1(\minus)\bigg)\bigg]
	\,,\\
	-k\alpha\,\Phi^{\textup{col},\textup{e}}(\alpha)
	&\equiv
	-k\alpha\, \E\Bigg[
	\log\bigg(1-\prod_{j=1}^k
	\bmeta_j(\minus)\bigg)\Bigg]
	+k\alpha\,\E\Bigg[\log
		\bigg(3-2\bhu(\plus)\bigg)\Bigg]
	+k\alpha\,\E
	\Bigg[\log\bigg(2-\bmeta_1(\minus)\bigg)\Bigg]\,.
	\end{align*}
Substituting these into \eqref{e:simplified.phi.col}
we see that $\Phi^\textup{col}(\alpha)$ agrees with
$\Phi(\alpha)$ from \eqref{e:phi.alpha}.
\end{proof}
\end{lem}

To conclude the proof of Proposition~\ref{p:first.moment.exponent},
in light of Corollary~\ref{c:first.moment.exponent}
it remains to relate 
$\bm{\Psi}_{\DD}(\omstar)$ to $\Phi^\textup{col}(\alpha)$. The value
$\bm{\Psi}_{\DD}(\omstar)$
is defined by the constrained optimization 
\eqref{e:nu.opt}.
As we already noted in the discussion following
Corollary~\ref{c:first.moment.exponent},
the optimal $\optdbh[\omstar]$ takes a simple form,
given explicitly by the $\dbhstar$ in Lemma~\ref{l:nu.star.as.optimizer.for.starpi}.
However, we also noted that the analogous
$\hbhstar$ (see \eqref{e:opt.clause.tuple.measure.star}) is \bemph{not} the same as $\opthbh[\omstar]$, simply because $\hbhstar$ is not consistent with the given marginals $\omstar$. 
However, we do have the following:

\begin{lem}\label{l:clause.tuple.measure.corrected} Let $\omega$ be as specified by Definition~\ref{d:judicious}. If $\bL$ is a stable clause type
(Definition~\ref{d:stable}), then there exists a measure $\hbh_{\bL}$
which is consistent with $\omega$ and satisfies,
for all $r\ge (\log k)^2$,
	\[\Ent(\hbh_{\bL}) \ge \Ent(\hbhstar_{\bL}) - 
\f1{k^{r/2}}\]

\begin{proof} Throughout this proof we fix a clause $a$ of type $\bL$, and suppress these from the notation when possible. The idea is to start from the explicit measure $\hbh_\star\equiv\hbhstar_{\bL}$ of \eqref{e:opt.clause.tuple.measure.star} --- which 
has incorrect edge marginals as commented above ---
and make small iterative adjustments to achieve the required edge marginals $\starpi$. Indeed, for all $e\in\delta a$, by the message-stability condition \eqref{e:message.stability} on $e$, we have
	\beq\label{e:conseq.of.message.stab}
	\f{\hbh_\star(\sigma_e=\sigma)}{\starpi_e(\sigma)}
	=\Bigg(\f{\tilde{q}_e(\sigma)}
		{\displaystyle \sum_\tau
		\dqstar_e(\tau) \tilde{q}_e(\tau)}\Bigg)
	\Bigg/\Bigg(
	\f{\hqstar_e(\sigma)}{\displaystyle \sum_\tau
		\dqstar_e(\tau) \hqstar_e(\tau)}\Bigg)
	= 1 + \f{O(1)}{k^r}
	\eeq
for all $\sigma\in\set{\RYGB}$. This shows that the edge marginals of $\hbh_\star$ are only slightly off from the canonical ones $\starpi$.\smallskip

\noindent\bemph{Step 1. Correct marginal proportions of $\SPIN{red}$ spins.} We first reweight $\hbh_\star$ to produce a measure $\hbh_0$ which has the correct marginal proportions of $\SPIN{red}$ spins. To this end, let $\SPIN{CY}$ denote the event that $\usi_{\delta a}$ has no $\SPIN{red}$ spins. It follows by using \eqref{e:indicator.of.valid.clause.coloring} and \eqref{e:opt.clause.tuple.measure.star} that for any $e'\ne e''$ in $\delta a$,
	\[
	\hbh_\star(\SPIN{CY})
	= 1-\sum_{e\in\delta a}\hbh_\star(\sigma_e=\red)
	\ge
	\hbh_\star\Big(\sigma_{e'}=\cya=\sigma_{e'}\Big)
	= \f{\dqstar_{e'}(\cya)\dqstar_{e''}(\cya)}
		{\bm{\hat{z}}_\star}
	\prod_{f\in\delta a \setminus\set{e',e''}}
	\dqstar_{f}(\cya,\yel)\,,
	\]
where $\bm{\hat{z}}_\star$ is the normalizing constant in \eqref{e:opt.clause.tuple.measure.star}. It is given by
	\beq\label{e:normalizing.constant.cancel}
	\bm{\hat{z}}_\star
	=\prod_{e\in\delta a} \dqstar_e(\cya,\yel)
	-\prod_{e\in\delta a} \dqstar_e(\yel)
	+ \overbrace{\bigg[ \sum_{e\in\delta a} 
		\Big(\dqstar_e(\red)-\dqstar_e(\cya)\Big)
		\prod_{e'\in\delta a\setminus e}
	\dqstar_{e'}(\yel)\bigg]}^\textup{zero}
	\le\prod_{e\in\delta a} \dqstar_e(\cya,\yel)
	\,,
	\eeq
where the term in square brackets vanishes by
the correspondence \eqref{e:color.recursions.eta}.
It also follows from \eqref{e:color.recursions.eta},
in combination with the message-stability condition
\eqref{e:message.stability},
that 
	\[
	\f{\dqstar_e(\cya)}{\dqstar_e(\cya,\yel)}
	=1- \stareta_e(\minus) \ge \f1{2^r}\,.
	\]
Combining these bounds gives, assuming $r\ge (\log k)^2$,
	\beq\label{e:stability.lbd.no.red.spins}
	\hbh_\star(\SPIN{CY})
	\ge
	\f{\dqstar_{e'}(\cya)}{\dqstar_{e'}(\cya,\yel)}
	\cdot
	\f{\dqstar_{e''}(\cya)}{\dqstar_{e''}(\cya,\yel)}
	\ge \f1{4^r}\,.
	\eeq
On the other hand, for any measure $\hbh$ that is consistent with $\starpi_e$ for all $e\in\delta a$, we must have
	\beq\label{e:stability.error.on.CY}
	\f{\hbh(\SPIN{CY})}{\hbh_\star(\SPIN{CY})}-1
	= \f1{\hbh_\star(\SPIN{CY})}
	\bigg(\sum_{e\in\delta a}
	\hbh_\star(\sigma_e=\red)
	 - \sum_{e\in\delta a}\starpi(\red)\bigg)
	= O\bigg(\f{4^r}{k^r}\bigg)\,,
	\eeq
where the last inequality follows by combining \eqref{e:conseq.of.message.stab} with \eqref{e:stability.lbd.no.red.spins}.
It follows by using \eqref{e:conseq.of.message.stab} and \eqref{e:stability.error.on.CY} that
	\beq\label{e:red.reweight.gamma.estimate}
	\gm_e
	\equiv\bigg(
	\f{\starpi_e(\red)}{\hbh(\SPIN{CY})}\bigg)
	\bigg/\bigg(\f{\hbh_\star(\sigma_e=\red)}
		{\hbh_\star(\SPIN{CY})}\bigg)
	= 1 + O\bigg(\f{4^r}{k^r}\bigg)\,,
	\eeq
for all $e\in\delta a$, where $\hbh$ again stands for any measure that is consistent with $\starpi$. Then define
	\[\hbh_0(\usi_{\delta a})
	\equiv \f1{\bm{\hat{z}}_\circ}
	\hbh_\star(\usi_{\delta a})
	\prod_{e\in\pd a}
	(\gm_e)^{\Ind{\sigma_e=\red}}\,,\]
where $\bm{\hat{z}}_\circ$ is the normalization constant that makes $\hbh_0$ a probability measure. We then calculate
	\[\bm{\hat{z}}_\circ
	=\hbh_\star(\SPIN{CY})
	+\sum_{e\in\delta a}
		\hbh_\star(\sigma_e=\red)\gamma_e
	= \f{\hbh_\star(\SPIN{CY})}{\hbh(\SPIN{CY})}\]
(using the definition of $\gamma$ with some simple algebra). It follows that
	\[
	\hbh_0(\sigma_e=\red)
	= \f{\hbh_\star(\sigma_e=\red) \gamma_e}
		{\bm{\hat{z}}_\circ}
	= \starpi_e(\red)\,,
	\]
for all $e\in\delta a$,
so $\hbh_0$ has the correct marginal proportions of $\SPIN{red}$ spins, as desired.\smallskip

\noindent\bemph{Step 2. Estimates for measure conditional on no $\SPIN{red}$ spins.} Note that the above reweighting does not change the measure conditional on $\SPIN{CY}$, that is to say, $\hat{\mu}_0(\cdot)\equiv\hbh_0( \cdot \,|\, \SPIN{CY})$ is the same as $\hbh_\star( \cdot \,|\, \SPIN{CY})$. It remains to correct the edge marginals conditional on $\SPIN{CY}$.
If $\hbh$ is any measure that is consistent with $\starpi$, we must have for $\sigma\in\set{\cya,\yel}$ that
	\[\hat{\mu}\Big(\sigma_e=\sigma\Big)\equiv
	\hbh\Big(\sigma_e=\sigma\,\Big|\,\SPIN{CY}\Big)
	=\f1{\hbh(\SPIN{CY})}
	\bigg(\starpi_e(\sigma)
	-\Ind{\sigma=\yel} 
	\sum_{e'\in\delta a\setminus e}\starpi_{e'}(\red)
	\bigg)
	\le \f{\starpi_e(\sigma)}{\hbh(\SPIN{CY})}
	\]
It follows using \eqref{e:conseq.of.message.stab},
\eqref{e:stability.lbd.no.red.spins}, and \eqref{e:stability.error.on.CY} that, for $\sigma\in\set{\cya,\yel}$,
	\beq\label{e:error.cond.on.CY}
	\hat{\mu}_0\Big(\sigma_e=\sigma\Big)
		-\hat{\mu}\Big(\sigma_e=\sigma\Big)
	=
	\hbh_\star\Big(
		\sigma_e=\sigma\,\Big|\,\SPIN{CY}\Big)
	-\hbh\Big(\sigma_e=\sigma\,\Big|\,\SPIN{CY}\Big)
	= O\bigg( \f{4^r}{k^r}\cdot \f{\starpi_e(\sigma)}
		{\hbh_\star(\SPIN{CY})}
	\bigg)
	= O\bigg( \f{16^r}{k^r}
	\bigg)\,.
	\eeq
Now, with respect to a given coloring $\usi_{\delta a}$, let us say that an edge $e\in\delta a$ is $\whi\equiv\SPIN{white}$ if at least two of the other edges in the clause are $\SPIN{cyan}$: that is, if
	\[\bigg|
	\Big\{
	e'\in\delta a\setminus e:\sigma_{e'}
	\in \set{\grn,\blu}\Big\}\bigg|\ge2\,.\]	
Again recalling \eqref{e:indicator.of.valid.clause.coloring}, the interpretation of $e$ being $\SPIN{white}$ is that $\sigma_e$ can be freely flipped between $\cya$ and $\yel$ without invalidating the coloring around $a$.
For each $e\in\delta a$, we can choose any $e'\ne e''$ in $\delta a\setminus e$ and lower bound
	\[
	W_e(\sigma)
	\equiv
	\f{\mu_0( \sigma_e=\sigma;\textup{$e$
	is $\SPIN{white}$} )}
	{\starpi_e(\sigma) / 
	\hbh_\star(\SPIN{CY})
	} \ge
	\f{\dqstar_e(\sigma)
	\dqstar_{e'}(\cya)\dqstar_{e''}(\cya)}
	{\starpi_e(\sigma) \bm{\hat{z}}_\star }
	\prod_{f\in\delta a \setminus\set{e,e',e''}} \dqstar_f(\cya,\yel)\,.\]
By a similar calculation as in \eqref{e:normalizing.constant.cancel}, we have
	\[
	\starpi_e(\sigma) \bm{\hat{z}}_\star
	\le
	\dqstar_e(\sigma) 
	\prod_{f\in\delta a \setminus e} \dqstar_f(\cya,\yel)\,.
	\]
It follows that, for $\sigma\in\set{\cya,\yel}$
and any choice of three distinct edges $e,e',e''\in\delta a$, we have
	\beq\label{e:enough.white.to.shift}
	W_e(\sigma)
	\ge 
	\f{\dqstar_{e'}(\cya)}{\dqstar_{e'}(\cya,\yel)}
	\cdot\f{\dqstar_{e''}(\cya)}{\dqstar_{e''}(\cya,\yel)}
	\ge\f1{4^r}\,.\eeq
Recalling the definition of $W_e(\sigma)$ above,
\eqref{e:enough.white.to.shift}
says that 
$\mu_0( \sigma_e=\sigma;\textup{$e$
	is $\SPIN{white}$})$ is large relative to \eqref{e:error.cond.on.CY}.\smallskip
	
\noindent\bemph{Step 3. Correct marginals given no $\SPIN{red}$ spins.} We can now define a sequence of measures $\hat{\mu}_0,\ldots,\hat{\mu}_{|\delta a|}$ where $\hat{\mu}_0$ is as defined above, and $\hat{\mu}_j$ will be consistent with $\starpi$ on the first $j$ edges in $\delta a$.
Given $\hat{\mu}_0$, we look at the first edge and define $\hat{\mu}_1$ to satisfying the following constraints:
\begin{enumerate}[(a)]
\item \label{i.not.white.no.shift} For any $\usi_{-1}$ that contains only one $\cya$ spin, $\hat{\mu}_1(\cya,\usi_{-1})
=\hat{\mu}_0(\cya,\usi_{-1})$;
\item \label{i.white.keep.cyan.plus.yellow.const}
For any $\usi_{-1}$ that contains at least two $\cya$ spins, $\hat{\mu}_1(\cya,\usi_{-1})
	+\hat{\mu}_1(\yel,\usi_{-1})
	=\hat{\mu}_0(\cya,\usi_{-1})
	+\hat{\mu}_0(\yel,\usi_{-1})$; 
\item \label{i.white.shift.cyan.and.yellow}
The appropriate amount of mass is shifted between $\set{\sigma_1=\cya}$ and $\set{\sigma_1=\yel}$:
	\[
	\hat{\mu}_1\Big(\sigma_{e_1}=\sigma; \textup{$e_1$
		is $\SPIN{white}$}\Big)
	=\hat{\mu}_0\Big(\sigma_{e_1}=\sigma; \textup{$e_1$
		is $\SPIN{white}$}\Big)
		+\bigg\{\hat{\mu}\Big(\sigma_e=\sigma\Big)
		-\hat{\mu}_0\Big(\sigma_e=\sigma\Big)\bigg\}\,.
	\]
\end{enumerate}
Constraints \eqref{i.not.white.no.shift} and \eqref{i.white.keep.cyan.plus.yellow.const} ensure that the marginals on the other edges are not affected. Finally, constraint~\eqref{i.white.shift.cyan.and.yellow} ensures that $\hat{\mu}_1$ will have marginals consistent with $\starpi$ on the first edge in $\delta a$, and it is feasible thanks to the estimates \eqref{e:error.cond.on.CY}~and~\eqref{e:enough.white.to.shift} from above. 
We can repeat the procedure with the remaining edges in $\delta a$ to arrive at $\hat{\mu}_{|\delta a|}\equiv\hat{\mu}$. From the construction and from \eqref{e:error.cond.on.CY} we have the very crude bound
	\beq\label{e:tvdistance.on.CY}
	\Big|\hat{\mu}(\usi_{\delta a})
	-\hat{\mu}_0(\usi_{\delta a})\Big|
	\le
	\sum_{e\in\delta a}
	\bigg|\hat{\mu}_0\Big(\sigma_e=\sigma\Big)
		-\hat{\mu}\Big(\sigma_e=\sigma\Big)\bigg|
	\le O\bigg( \f{16^rk}{k^r}\bigg)\,,
	\eeq
for all $\usi_{\delta a}\in\SPIN{CY}$.
 We then finally define a measure on all valid colorings of $\delta a$ by
	\[
	\hbh(\usi_{\delta a})
	\equiv \hbh_0\Big(\usi_{\delta a} ;
		\usi_{\delta a}\notin \SPIN{CY}\Big)
	+\hbh_0(\SPIN{CY}) \hat{\mu}(\usi_{\delta a})\,.\]
By construction, $\hbh$ is consistent with the marginals $\starpi$, and satisfies
	\[\Big|
	\hbh(\usi_{\delta a})
	-\hbh_\star(\usi_{\delta a})\Big|
	\le
	O\bigg( \f{16^rk}{k^r}\bigg)
	\]
for all $\usi_{\delta a}$, by combining \eqref{e:tvdistance.on.CY} with our earlier estimate \eqref{e:red.reweight.gamma.estimate} on the weights $\gamma_e$. This gives
	\[
	\Ent(\hbh) \ge \Ent(\hbh_\star) - \f1{k^{r/2}}\,,
	\]
concluding the proof of the lemma.
\end{proof}
\end{lem}

\begin{proof}[Proof of Proposition~\ref{p:first.moment.exponent}]
We have from Corollary~\ref{c:first.moment.exponent} that
	\[\E_{\DD}\ZZ = \exp\bigg\{
		 n\bm{\Psi}_{\DD}(\omstar) + o(n)\bigg\}\,,\]
where $\bm{\Psi}_{\DD}$ is defined by \eqref{e:nu.opt}. Let $\bh\equiv(\dbhstar,\hbh)$ for $\dbhstar$ 
as in Lemma~\ref{l:nu.star.as.optimizer.for.starpi},
and $\hbh$ as given by Lemma~\ref{l:clause.tuple.measure.corrected}. Then, by construction, $\nu$ is consistent with marginals $\omstar$, so $\ZZ\ge\ZZ[\bh]$. It follows by Lemma~\ref{l:combinatorial.mmt.simplification} that
	\[
	\E_{\DD}\ZZ
	\ge\E_{\DD}\ZZ[\bh]
	= \exp\bigg\{
		 n\bm{\Phi}_{\DD}(\bh) + o(n)\bigg\}\,,
	\]
for $\bm{\Phi}_{\DD}(\bh)$ as defined by \eqref{e:rate.function.given.gen.degseq}.
Now let $\bhstar\equiv(\dbhstar,\hbhstar)$
for $\hbhstar$ as defined by \eqref{e:opt.clause.tuple.measure.star}.
Then Lemma~\ref{l:clause.tuple.measure.corrected} gives
	\[
	\bm{\Phi}_{\DD}(\bh)
	\ge
	\bm{\Phi}_{\DD}(\bhstar)
	- \f1{k^{r/3}}\,.
	\]
In the limit $n\to\infty$ followed by $R\to\infty$, the $R$-neighborhood empirical measure $\DD$
concentrates around the limiting Galton--Watson measure, using the fact that an $o_R(1)$ fraction of variables are removed by preprocessing (Proposition~\ref{p:small.fraction.removed.in.processing}). It then follows by the stability results
(Lemma~\ref{l:stability}) that $\bhstar$ converges in the distributional sense to the measures on the limiting random tree. It follows that
	\[
	\lim_{R\to\infty}\bigg[
	\lim_{n\to\infty}
	\bm{\Phi}_{\DD}(\bhstar)\bigg]
	=\Phi^\textup{col}(\alpha)
	\]
in probability, and the result follows since
$\Phi^\textup{col}(\alpha)=\Phi(\alpha)$
by Lemma~\ref{l:phi.equals.phi.col}.
\end{proof}

\section{Analysis of preprocessing}\label{s:processing}

\noindent In this section we prove Propositions~\ref{p:small.fraction.removed.in.processing}--\ref{p:unif}. The section is organized as follows:
\begin{enumerate}[--]
\item In \S\ref{ss:bootstrap.defects} we bound the occurrence of \bemph{defective} variables (Definition~\ref{d:j.defective}) in the $\uPGW$ random tree (Proposition~\ref{p:sparse.subtree.with.many.defects}). We already have Proposition~\ref{p:uPGW.is.J.robust} from Section~\ref{s:onersb} which bounds the probability that a variable fails to be \bemph{nice}. The main task of \S\ref{ss:bootstrap.defects} is to control the bootstrap percolation of non-nice variables that creates defective variables.

\item In \S\ref{ss:orderly.contained} we continue to analyze local properties of variables in the $\uPGW$ random tree. To this end, recall that the \bemph{self-contained} (Definition~\ref{d:contained}) and \bemph{orderly} (Definition~\ref{d:orderly}) properties both refer to being near relatively few defects (although each property captures a different notion of ``relatively few''). In \S\ref{ss:orderly.contained} we apply the estimates of \S\ref{ss:bootstrap.defects} to bound the occurrence of variables that fail to be self-contained or orderly (Propositions~\ref{p:path.intersect.not.selfcont}~and~\ref{p:path.intersect.not.orderly}). An immediate consequence of these two propositions is that most variables are \bemph{perfect}, i.e., both self-contained and orderly. Now recall from Definition~\ref{d:perfect.fair} that a variable is \bemph{fair} if it is stable, it satisfies a loose bound on the volume of its local neighborhood, and it does not lie near many imperfect variables --- thus we can deduce from the combination of Propositions~\ref{p:one.stable}, \ref{p:path.intersect.not.selfcont}~and~\ref{p:path.intersect.not.orderly} that most variables are fair (Corollary~\ref{c:fair}). In fact we can further deduce (Corollary~\ref{c:excellent}) that most variables are \bemph{excellent} --- from Definition~\ref{d:good.exc}, excellent roughly means fair in a strong sense. Since \bemph{good} is weaker than excellent, it immediately implies that most variables are good. 

\item In \S\ref{ss:preprocessing.structural} and \S\ref{ss:preprocessing.probabilistic} we prove Proposition~\ref{p:small.fraction.removed.in.processing}, which controls the set of variables $\bsp'(A;\GG)$ that is removed by the preprocessing algorithm (Definition~\ref{d:proc}). Recall $A$ is the set of variables in $\GG$ that are improper (Definition~\ref{d:simple.types}) or not $1$-good (Definition~\ref{d:good.exc}) --- it is easy to bound the occurrence of improper variables, and the bound on variables that are not $1$-good comes directly from Corollary~\ref{c:excellent} in \S\ref{ss:orderly.contained}. Thus we already know from \S\ref{ss:orderly.contained} that variables are unlikely to be in $A$, and the task of \S\ref{ss:preprocessing.structural} and \S\ref{ss:preprocessing.probabilistic} is to deduce a bound on $\bsp'(A;\GG)$. To this end, in \S\ref{ss:preprocessing.structural} we prove a deterministic result (Proposition~\ref{p:corrupted.tree.in.graph}), saying that any large connected component of $\bsp'(A;\GG)$ must have certain characteristics. In \S\ref{ss:preprocessing.probabilistic} we show that those characteristics rarely occur in the random graph, thereby concluding the proof of Proposition~\ref{p:small.fraction.removed.in.processing}.

\item In \S\ref{ss:positive.fraction.combinatorial} ~and~\S\ref{ss:positive.fraction.prob} we prove Proposition~\ref{p:posfrac}, which says that, with high probability over the random graph $\GG$, each surviving total type occurs linearly many times in the processed graph $\proc\GG$.
To this end, in \S\ref{ss:positive.fraction.combinatorial} we prove a deterministic result (Corollary~\ref{c:event.implies.desired.type}) which says that if the local neighborhood of a clause $a$ in $\GG$ satisfies certain properties, then $a$ will have a particular total type $\bL$ in the processed graph $\proc\GG$. Since the conditions to have any given total type $\bL$ are essentially local, in \S\ref{ss:positive.fraction.prob} we are able to argue that the total number of clauses satisfying the local conditions for type $\bL$ concentrates around its mean, which grows linearly in $n$.

\item Lastly, in \S\ref{ss:unif} we prove Proposition~\ref{p:unif} which says that the processed graph $\proc\GG$ is uniform given its processed neighborhood profile $\DD$, in the sense discussed in Remark~\ref{r:processed.graph.types.CM} .
\end{enumerate}

\subsection{Bootstrap percolation of defects}
\label{ss:bootstrap.defects}

Recall from Definition~\ref{d:uPGW} that $\tree$ denotes
a $\uPGW$ tree, rooted at a variable $\vrt$. For $\II\in\set{0,1}$ we can use Definitions~\ref{d:j.stable}~and~\ref{d:j.defective}
to define the following subsets of $\tree$:
	{\setlength{\jot}{0pt}\begin{align*}
	D^{*,\II}
	\equiv D^{*,\II}(\tree)
	&\equiv \set{\textup{variables in $\tree$ that 
	are not $\II$-nice}} \\
	\subseteq
	D^{\KAPPA,\II}
	\equiv D^{\KAPPA,\II}(\tree)
	&\equiv \set{\textup{variables in $\tree$ within
		distance $\KAPPA$ of $D^{*,\II}$}}\\
	\subseteq
	\DEFECTIVE^\II
	\equiv
	\DEFECTIVE^\II(\tree)
	&\equiv \set{\textup{$\II$-defective variables
	in $\tree$}}.
	\end{align*}}%
The main goal of this subsection is to bound the occurrence of the $\II$-defective set $\DEFECTIVE^\II$ in $\tree$. More precisely, we will bound the intersection of $\DEFECTIVE^\II$ with \bemph{sparse subtrees} of $\tree$. Let
	\beq\label{e:bold.Lambda.abstract.defn}
	\bLambda_{\CC,\ell}
	\equiv
	\left\{ \begin{array}{c}
	\textup{bipartite factor trees 
	$T\equiv(V_T,F_T,E_T)$,
	rooted at a variable $\vrt$,}\\
	\textup{with maximum degree at most $\CC$,
		and number of variables $|V_T|=\ell$}
	\end{array}
	\right\}\,.\eeq
For any variable-rooted tree $\tree$ we let
	\beq\label{e:bold.Lambda.within.PGW.defn}
	\bLambda_{\CC,\ell}(\tree)
	\equiv
	\left\{ \begin{array}{c}
	\text{bipartite factor trees 
	$T\equiv(V_T,F_T,E_T)$
	with $\vrt \in T \subseteq \tree$,}\\
	\textup{maximum degree at most $\CC$,
	and number of variables $|V_T| = \ell$}
	\end{array}
	\right\}\,.\eeq
Without loss of generality we always assume $\CC\ge2$.
When considering $T\subseteq\tree$, for any vertex $x\in\tree$ we will always use $\pd x$ to denote the immediate neighbors of $x$ in $\tree$,
and $N(x)\equiv \pd_1 x$ the set of vertices at unit distance from $x$. We then denote
	{\setlength{\jot}{0pt}\begin{align*}
	\pd_T x &\equiv T \cap \pd x\,,\\
	N_T(x) &\equiv T\cap N(x)\,.
	\end{align*}}%
If $T\in\bLambda_{\CC,\ell}(\tree)$, then by definition we have $|\pd_T x| \le \CC$ for all $x\in V_T\cup F_T$,
so $N_T(x) \le \CC(\CC-1) \le \CC^2$. The main result of this subsection is the following:

\begin{ppn}\label{p:sparse.subtree.with.many.defects}
Let $\II\in\set{0,1}$, and suppose $\CC\ge2$ is upper bounded by an absolute constant. Then, as long as $k$ exceeds an absolute constant $k_0$ (depending only on the bound on $\CC$), we have the bound
	\[
	\uPGW\bigg( |\DEFECTIVE^\II \cap V_T| 
	\ge \ell\ep
	\text{ for some }
	T\in\bLambda_{\CC,\ell}\bigg)
	\le \f{\exp(k^2\ell)}{ \exp(
		2^{k/4}\ell\ep /k^2 ) }
	\]
for all $\ep\ge0$ and all integer $\ell\ge0$.
\end{ppn}

\begin{lem}\label{lem-local-D-0}
Let $\II\in\set{0,1}$, and suppose $\CC\ge2$ is upper bounded by an absolute constant. Then, as long as $k$ exceeds an absolute constant $k_0$ (depending only on the bound on $\CC$), we have the bound
	\[
	\uPGW\bigg( |D^{\KAPPA,\II} \cap V_T| 
	\ge \ell\ep
	\text{ for some }
	T\in\bLambda_{\CC,\ell}\bigg)
	\le \f{ \exp( k^2\ell)}
	{ \exp( 2^{k/4} \ell\ep) }
	\]
for all $\ep>0$ and all integer-valued $\ell\ge0$.

\begin{proof}
We divide the proof into three steps below. In the first step, we show that if some sparse subtree $T\subseteq\tree$ has a nontrivial intersection with $D^{\KAPPA,\II}$, then there must exist $T\subseteq \tprime \subseteq\tree$ such that $\tprime$ is also sparse, and has a nontrivial intersection with $D^{*,\II}$ --- this step is deterministic.
In the second step, we explain a general identity concerning the expected number of subtrees of $\tree\sim\uPGW$ satisfying a given property.
In the third step we bound this quantity to show that $\tprime$ is unlikely to occur. If $\ell\ep=0$ the bound is vacuous, so we assume without loss that $\ell\ep\ge1$. \smallskip

\noindent\bemph{Step 1. Construction of $\tprime$.}
Take any $T \in \bLambda_{\CC,\ell}(\tree)$ such that $A=D^{\KAPPA,\II}\cap V_T$ has size $|A|\ge\ell\ep$.
For each variable $u\in A$, the number of variables 
in the $(4\KAPPA-1)$-neighborhood of $u$ in $T$ is at most
	\[
	|B_{4\KAPPA-1}(u) \cap V_T|
	=\sum_{\ell=0}^{4\KAPPA-1} \CC^{2\ell}
	= \f{\CC^{8\KAPPA}}{\CC^2-1}
	\le \CC^{8\KAPPA}\,.
	\]
Therefore we can extract $A'\subseteq A$ of size $|A'| = \lceil \ell\ep/\CC^{8\KAPPA}\rceil$ such that variables in $A'$ lie at pairwise distance at least $4\KAPPA$. By definition, since $A'\subseteq A \subseteq D^{\KAPPA,\II}$, every variable in $A'$ lies within distance $\KAPPA$ of $D^{*,\II}$.
For each $u\in A'$, let $g(u)$ be any of the variables in $D^{*,\II}$ which lies closest to $u$, and let $\gamma(u)$ be the path between $u$ and $g(u)$. The paths $\gamma(u)$, $u\in A'$, are of length at most $\KAPPA$, and are mutually disjoint. Let $\tprime$ be the union of $T$ with the paths $\gamma(u)$. Then we have $\tprime\in\bLambda_{\ell',\CC+1}(\tree)$ for 
	\[
	\ell \le \ell' \le 
	\ell + \bigg\lceil \f{\ell \ep}{\CC^{8\KAPPA}}
		 \bigg\rceil (\KAPPA-1)
		\le\ell\KAPPA\,.
	\]
Let $B'\equiv g(A')\subseteq D^{*,\II}$, and note that $g:A'\to B'$ is a one-to-one mapping. We have thus shown that if there exists $T\in\bLambda_{\CC,\ell}(\tree)$ with at least $\ell\ep$ variables in $D^{\KAPPA,\II}$,
then there must exist $\tprime\in\bLambda_{\CC+1,\ell'}(\tree)$, for $\ell\le \ell'\le \ell\KAPPA$,
and a subset $B'\subseteq D^{*,\II}\cap V_{\tprime}$
such that $B'=\lceil \ell\ep/\CC^{8\KAPPA} \rceil$,
and variables in $B'$ lie at pairwise distance at least $2\KAPPA$.\smallskip

\noindent\bemph{Step 2. Expected number of embedded subtrees.} 
Now fix any tree $\tprime\in\bLambda_{\CC+1,\ell'}$
for $\ell\le\ell'\le\ell\KAPPA$, and a subset
$B'$ of variables in $\tprime$ such that $|B'|= \lceil \ell\ep/\CC^{8\KAPPA} \rceil$, and variables in $B'$ lie at pairwise distance at least $2\KAPPA$. For this fixed pair $(\tprime,B')$ we will bound the probability
	\begin{align}\nonumber
	P(T',B') 
	&\equiv
	\uPGW\bigg(
	\textup{there exists an embedding
	$\zeta:T'\hookrightarrow\tree$ such that
	$\zeta(B')
	\subseteq D^{*,\II}(\tree)$}
	\bigg)
	\\
	&\le \int
	\bigg|\bigg\{
	\textup{embeddings }
	\zeta : \tprime\hookrightarrow\tree
	\textup{ such that }
	\zeta(B')
	\subseteq D^{*,\II}(\tree)
	\bigg\}
	\bigg|
	\,d\uPGW(\tree)\,.
	\label{e:expected.number.of.embeddings}
	\end{align}
(Later we will enumerate over $(\tprime,B')$.) To calculate \eqref{e:expected.number.of.embeddings}, it is useful to take the following (equivalent) view of the random tree $\tree\sim\uPGW$: to each variable $v$ we attach $D(v)$ which is an independently sampled Poisson point process of unit rate on the interval $[0,\alpha k]$. The atoms of $D(v)$ correspond to the child clauses of $v$, and we write them in ascending order as
	\[
	D(v)= \bigg(
	a_i(v) : 1\le i\le |D(v)|
	\bigg)\,,\quad
	0\le a_1(v) \le \ldots \le a_{|D(v)|}(v) \le \alpha k
	\,.
	\]
Meanwhile, to each clause $a$ we attach $D(a)=\set{1,\ldots,k-1}$ to indicate the $k-1$ child variables of $a$. In this view, 
an embedding $\zeta:\tprime\hookrightarrow\tree$ can be specified by giving $D_{\tprime}(v)\subseteq D(v)$ and $D_{\tprime}(a)\subseteq D(a)$ for the variables and clauses of $\tprime$.
Moreover, if $D$ is a Poisson point process of unit rate on $[0,\alpha k]$, and $D_\circ$ is a finite subset of $[0,\alpha k]$, then (by standard properties of Poisson point processes) the law of $D$ conditioned to contain $D_\circ$ is the same as the law of $D\cup D_\circ$. The volume (under Lebesgue measure) of all possible $D_{\tprime}(v)$ is given by
	\[\Bigg|\Bigg\{
	\bigg(a_i : 1\le i\le |D_{\tprime}(v)|\bigg) :
	0\le a_1 \le \ldots \le a_{|D_{\tprime}(v)|}
	\le \alpha k\Bigg\}\Bigg|
	= \f{(\alpha k)^{|D_{\tprime}(v)|}}{|D_{\tprime}(v)|!}\,.
	\]
The volume (under counting measure) of all possible $D_{\tprime}(a)$ is given by
	\[\Bigg|
	\bigg(v_j : 1\le j\le D_{\tprime}(a)\bigg)
	: v_j \textup{ integer-valued},
	1\le v_1 \le \ldots \le v_{|D_{\tprime}(a)|} \le k
	\Bigg|
	= \binom{k}{|D_{\tprime}(a)|}\,.
	\]
It follows that the expected volume of all embeddings
$\zeta:\tprime\to\tree$ is
	\begin{align}\nonumber
	\embed(\tprime) &\equiv \int
	\bigg|\bigg\{
	\textup{embeddings }
	\zeta : \tprime\hookrightarrow\tree
	\bigg\}
	\bigg|
	\,d\uPGW(\tree)\\
	&=\prod_{v\in V_{\tprime}} 
	\f{(\alpha k)^{|D_{\tprime}(v)|}}{|D_{\tprime}(v)|!}
	\prod_{a\in F_{\tprime}}
	\binom{k}{|D_{\tprime}(a)|}
	\le (\alpha k)^{|F_{\tprime}|}
	k^{|V_{\tprime}|-1} 
	\le (\alpha k^2)^{\ell'\CC}\,.
	\label{e:lebesgue.vol.of.embeddings}
	\end{align}
The law of $\tree$ conditioned on any such embedding is equivalent to its law under the measure $\uPGW(\tprime)$ described in Definition~\ref{d:PGW.based.on.T}. The expectation 
\eqref{e:expected.number.of.embeddings} is then equal to
	\beq\label{e:embeddings.times.embedded.probab}
	\embed(\tprime) \cdot
	\Big[\uPGW(\tprime)\Big]\bigg(
	B' \subseteq D^{*,\II}
	\bigg)
	\eeq
where $\embed(\tprime)\le (\alpha k^2)^{\ell'\CC}\le (\alpha k^2)^{\ell\KAPPA\CC}$.\smallskip

\noindent\bemph{Step 3. Probability bound for $\tprime$.} We now bound the last term in \eqref{e:embeddings.times.embedded.probab},
recalling that $\tprime\in\bLambda_{\CC+1,\ell'}$
and $B'\subseteq V_{\tprime}$
with $|B'|=\lceil \ell\ep/\CC^{8\KAPPA}\rceil$. The main difficulty is that the neighborhoods $B_R(u)\subseteq\tree$, for $u\in B'$, need not be disjoint, since we only ensure that variables in $B'$ have pairwise distance at least $2\KAPPA$ for $\KAPPA$ an absolute constant. This can be addressed by a simple modification: for $u\in B'$, consider $B_R(u)$ as a tree rooted as $u$. In this $u$-rooted tree, take each $w\in \tprime\cap \pd_2 u$ and delete the subtree descended from it (without deleting $w$ itself). Denote the result $B_R(u)^-\subseteq B_R(u)$. The modified neighborhoods $B_R(u)^-$, for $u\in B'$, are mutually disjoint. Define the depth-two subtree $T_u \equiv \tprime \cap B_2(u)$, and note that $B_R(u)^-$ has the same law as the random tree $\tree_R(u)^-$ defined as follows: first sample $\tree'\sim\uPGW(T_u)$ (rooted at $u$), then form $\tree''$ by deleting from $\tree'$ the subtrees of depth-two variables in $T_u$, then form $\tree_R(u)^-$ by taking the $R$-neighborhood of the root of $\tree''$. If the root of $\tree'$ is $(\CC+1)^4$-robust, then $\tree_R(u)^-$ must be $1$-nice. It follows that
	\beq\label{e:product.over.disjoint.nbhds.around.sparse.tree}
	\Big[\uPGW(\tprime)\Big]\bigg(
	B' \subseteq D^{*,\II}
	\bigg)
	\le
	\prod_{u\in B'}
	\Big[\uPGW(T_u)\Big]
	\bigg(\textup{root of $\tree'$
	is not $(\CC+1)^4$-robust}\bigg)
	\le 
	\f1{\exp(\Omega(|B'| k2^{k/4}))}\eeq
--- in the intermediate expression we use $\tree'$ to denote a random variable with law $\uPGW(T_u)$, and the last bound is by Proposition~\ref{p:uPGW.is.J.robust}. We can substitute this last bound into 
\eqref{e:embeddings.times.embedded.probab},
which we recall is equal to \eqref{e:expected.number.of.embeddings}. Enumerating over all pairs $(\tprime,B')$ 
(such that $\tprime\in\bLambda_{\CC+1,\ell}$ for $\ell\le \ell'\le \ell\KAPPA$, and 
$B'\subseteq V_{\tprime}$
with $|B'|=\lceil \ell\ep/\CC^{8\KAPPA}\rceil$)
gives
	\begin{align*}
	&\sum_{\tprime,B'}
	\int
	\bigg|\bigg\{
	\textup{embeddings }
	\zeta : \tprime\hookrightarrow\tree
	\textup{ such that }
	\zeta(B')
	\subseteq D^{*,\II}(\tree)
	\bigg\}
	\bigg|
	\,d\uPGW(\tree)\\
	&\le 
	\sum_{\ell\le \ell'\le\ell\KAPPA}
	\f{ (\CC+2)^{\ell'(\CC+2)}
		2^{\ell'}
		(\alpha k^2)^{\ell'\CC}
	}{\exp(
	\Omega(
	\lceil\ell\ep/\CC^{8\KAPPA}\rceil
	\cdot
	k2^{k/4}))}
	\le \f{ \exp( \ell \KAPPA \CC k)}
	{ \exp(
	2^{k/4}\ell\ep 
	 ) }\,,
	\end{align*}
for $k$ exceeding an absolute constant. The claimed bound follows.
\end{proof}
\end{lem}

\begin{lem}\label{l:BSP.in.tree.bounded.expansion}
Let $\tree$ be any bipartite factor tree in which all clauses have degree at most $k$. Then
	\[
	|A| 
	\ge \f{|\bsp(A;\tree)|}{2k}
	\]
for any finite subset $A$ of variables in $\tree$.

\begin{proof} Let $A_0\equiv A$, and for $t\ge1$ we let $A_t$ be the union of $A_{t-1}$ with all variables $u\in\tree$ with $|N(u)\cap A_{t-1}|\ge2$. Thus $A=A_0\subseteq A_1\subseteq \ldots$, and
(as in \eqref{eq-bootstrap-percolation}) we define
	\[
	\bsp(A;\tree)
	=\bigcup_{t\ge0} A_t\,.
	\]
For all $t\ge0$, let $G_t\equiv (A_t,F_t,E_t)$ be the subgraph of $\tree$ induced by $A_t$ --- that is, $F_t$ is the set of all clauses in $\tree$ having at least two incident variables in $A_t$, and $E_t$ is the set of all edges in $\tree$ between $A_t$ and $F_t$. As long as $t$ is finite, $G_t$ is a finite subgraph of $\tree$.
If we denote the (maximal) connected components of $G_t$
by $G_{t,i}\equiv (A_{t,i},F_{t,i},E_{t,i})$
for $1\le i\le i(t)$, and let $A_i\equiv A \cap G_{t,i}$, then we must have $\bsp(A_i; G_{t,i})=A_{t,i}$. If we can show for all $i$ that $|A_{t,i}|\le 2k|A_i|$, then summing over $i$ gives
 	\[
	|A_t|
	=\sum_{i=1}^{i(t)} |A_{t,i}|
	\le 2k\sum_{i=1}^{i(t)} |A_i|
	= 2k|A|\,.
	\]
If the bound holds for all $t$, then the conclusion of the lemma follows. Since each $G_{t,i}$ is a finite tree, we conclude that it suffices to show the following special case of the lemma: for any finite bipartite factor tree $\tree=(V,F,E)$ in which all clauses have width at most $k$, if $A\subseteq V$ such that $\bsp(A;\tree)=V$, then we must have $|V| \le 2k|A|$.

To show the last assertion, for all $v\in V$ let $\tau(v)\equiv \min\set{j\ge0: v\in A_j}$. For each edge $e=(av)\in E$, let
	\[
	m(av)
	\equiv\min\bigg\{ 2,
	\Big|\Big\{ u\in\pd a:
		\tau(u)<\tau(v) \Big\}\Big|
	\bigg\}\in\set{0,1,2}\,.
	\]
Since $\bsp(A;\tree)=V$, the definition of $\bsp$ implies that each variable $v\in V$ must have
	\[\sum_{a\in\pd v} m(av)\ge2\,.\]
On the other hand, for each clause $a\in F$, if we reorder the variables in $\pd a$ as $(v_1,\ldots,v_{|\pd a|})$ such that $\tau(v_i)$ is nondecreasing in $i$, then we must have 
$m(av_1)=0$ and $m(av_2)\in\set{0,1}$, so
	\[
	\sum_{v\in\pd a}m(av) 
	\le 0 + 1 + (|\pd a|-2)2 = 2|\pd a|-3\,.
	\]
Combining these inequalities gives
	\[
	2|V\setminus A|
	\le
	\sum_{v\in V} \sum_{a\in\pd v} m(av)
	= \sum_{a\in F} \sum_{v\in\pd a} m(av)
	\le \sum_{a\in F}\Big(2|\pd a|-3\Big)
	=2|E|-3|F|\,.
	\]
Since $\tree=(V,F,E)$ is by assumption a finite tree where all clauses have degree at most $k$, we must have
	\[
	|V|-1 = |E|-|F|
	=\sum_{a\in F}\bigg(|\pd a|-1\bigg)
	\le (k-1) |F|\,.
	\]
Substituting into the previous gives
	\[
	2|V\setminus A|
	\le 2\Big(|V|-|1|\Big)-|F|
	\le \bigg(2-\f1{k-1}\bigg)
		\Big(|V|-1\Big)\,.
	\]
Rearranging gives
	\[
	|A| \ge 1 + \f{|V|-1}{2(k-1)}
	\ge \f{|V|}{2k}\,,
	\]
which proves the claim. The lemma follows as discussed above.
\end{proof}
\end{lem}

\begin{lem}\label{l:larger.tree.intersecting.bsp.init.set}
Let $\tree$ be any bipartite factor tree in which all clauses have degree at most $k$.
For any $\CC\ge2$ and $\ell\ge1$,
and any $T\in\bLambda_{\CC,\ell}(\tree)$,
there is a tree $T'\in\bLambda_{\CC',\ell'}(\tree)$, with $\CC'=\CC+2$ and $\ell\le \ell'\le \ell(\CC')^{R'}$, such that
	\[
	|D^{\KAPPA,\II}\cap V_{T'}|
	\ge
	\f1{2k}
	\bigg(
	|\DEFECTIVE^\II \cap V_T|
	+|V_{T'}\setminus V_T|
	\bigg)
	\]
for $\II\in\set{0,1}$.

\begin{proof}
Recall Definition~\ref{d:j.defective} that $v \in \DEFECTIVE^\II$ if and only if $v\in\bsp( D^{\KAPPA,\II}\cap B_{R'/2}(v);B_{R'/2}(v))$. Let $\tree'\equiv B_{R'/2}(T)$, that is, the union of the $(R'/2)$-neighborhoods of all variables in $T$. Let $A_0\equiv D^{\KAPPA,\II} \cap \tree'$,
and for $t\ge1$ let $A_t$ be the union of $A_{t-1}$ with all variables in $\tree'$ that have at least two neighbors in $A_{t-1}$. Then let
	\[
	A_\infty
	\equiv \bigcup_{t\ge0} A_t
	= \bsp( D^{\KAPPA,\II} \cap \tree'; \tree').
	\]
Let $G_t$ be the subgraph of $\tree'$ induced by $A_t$. We will construct a sequence 
of trees $T_s\equiv(V_s,E_s,F_s)\subseteq G_\infty$ (terminating at $T'$), as follows.
Let $T_0\equiv T$.
For all variables $u\in A_\infty$, let $\tau(u)\equiv \min\set{t\ge0: u\in A_t}$. For $s\ge0$, let
	\[
	U(T_s)
	\equiv
	\bigg\{
	u\in V_s \,\bigg|
	\begin{array}{c} 
	1\le \tau(u)<\infty, 
	\textup{ and $u$ does not have at least}\\
	\textup{two neighboring variables 
	in $T_s$ with smaller $\tau$}
	\end{array}
	\bigg\}\,.
	\]
If $U(T_s)=\emptyset$, we terminate the process and set $T_s=T_\infty$. Otherwise, take any $u\in U(T_s)$: it must have at least one neighboring variable $u'$ that does not lie in $T_s$ and has $\tau(u')<\tau(u)$. Among all such $u'$, choose one with minimal $\tau$. Let $T_{s+1}$ be the graph induced by $V_s \cup \set{u'}$. In the resulting $T_\infty\equiv T'$, any vertex that did not belong to $T$ will have degree at most $3$, while any vertex that did belong to $T$ will have its degree in $T'$ at most two larger than its degree in $T$. It follows that 
$T'\in\bLambda_{\CC',\ell'}(\tree)$, with $\CC'=\CC+2$ and $\ell\le \ell'\le \ell(\CC')^{R'}$. It is straightforward to check (e.g.\ by induction) that the above construction implies
	\[(\DEFECTIVE^\II \cap V_T)
	\cup (V_{T'} \setminus V_T)
	\subseteq
	\bsp(D^{\KAPPA,\II}\cap V_{T'};T')\,.
	\]
Combining with Lemma~\ref{l:BSP.in.tree.bounded.expansion} gives
	\[
	\f1{2k}
	\bigg(
	|\DEFECTIVE^\II \cap V_T|
	+|V_{T'}\setminus V_T|
	\bigg)
	\le 
	\f{\bsp(D^{\KAPPA,\II}\cap V_{T'}; T')}{2k}
	\le|D^{\KAPPA,\II}\cap V_{T'}|\,,
	\]
as claimed.\end{proof}\end{lem}

\begin{proof}[Proof of Proposition~\ref{p:sparse.subtree.with.many.defects}]
It follows from Lemma~\ref{l:larger.tree.intersecting.bsp.init.set} that
	\begin{align*}
	&\uPGW\bigg( |\DEFECTIVE^\II \cap V_T| 
	\ge \ell\ep
	\text{ for some }
	T\in\bLambda_{\CC,\ell}\bigg)\\
	&\le
	\sum_{\ell'=\ell}^{\lceil \ell(\CC')^{R'}\rceil}
	\uPGW\bigg( |D^{\KAPPA,\II} \cap V_{T'}| 
	\ge \f{\ell\ep +(\ell'-\ell)}{2k}
	\text{ for some }
	T'\in\bLambda_{\CC',\ell'}\bigg)\,.
	\end{align*}
It follows from Lemma~\ref{lem-local-D-0} that the last expression is upper bounded by
	\[
	\sum_{\ell'=\ell}^{\lceil \ell(\CC')^{R'}\rceil}
	\f{\exp( k^2[\ell + (\ell'-\ell)] )}
	{\exp(2^{k/4} [ \ell\ep +(\ell'-\ell) ] /(2k))}
	\le \f{O(1) \exp(k^2\ell)}
	{\exp (\ell \ep 2^{k/4} /(2k))}
	\le \f{\exp(k^2\ell)}{ \exp(\ell\ep
		2^{k/4} /k^2 ) }\,,
	\]
concluding the proof.
\end{proof}

Up to this point we have proved bounds
for the $\uPGW(T)$ measures (for instance, Proposition~\ref{p:uPGW.is.J.robust}) by direct analysis --- the basic intuition being that, when $T$ is a sparse tree, the $\uPGW(T)$ measure is not so different from the $\uPGW$ measure. We now make this more precise by proving a general bound which allows us to more easily transfer bounds from $\uPGW$ to $\uPGW(T)$:

\begin{lem}\label{l:radon.derivative.pgwT}
Let $T$ be any fixed variable-rooted tree of maximum degree $\CC$,
where $\CC$ is an absolute constant.
Then the Radon--Nikodyn derivative between the measures $\uPGW(T)$ and $\uPGW$ satisfies the second moment bound 
	\[
	\E\bigg[\bigg( \f{d\uPGW(T)}{d\uPGW}
		\bigg)^2\bigg] \le 2\,,
	\]
where $\E$ denotes expectation under $\uPGW$.
(The bound holds for $k\ge k_0$ where $k_0$ depends only on $\CC$.)

\begin{proof}
According to the original definitions, in a random tree $\tree\sim\uPGW$
or $\tree\sim\uPGW(T)$, only the root variable is distinguished, while all other variables are unlabelled. For the purposes of this proof, however, we now instead consider both $\tree\sim\uPGW$ and $\tree\sim\uPGW(T)$ as rooted \emph{labelled} trees, where the children of each vertex are ordered uniformly at random. Separately, we also assume that $T$ comes with a fixed labelling of its vertices, so that we may speak of embeddings
$T\hookrightarrow\tree$.\smallskip

\noindent\bemph{Step 1. Martingale of Radon--Nikodym derivatives.} If $\uPGW(T)\ll\uPGW$, then the labellings discussed above do not affect the value of the Radon--Nikodym derivative, which we hereafter write as
	\[\RN 
	\equiv \f{d\uPGW(T)}{d\uPGW}(\tree)
	\equiv \RN(\tree;T) \,.\]
To show that $\uPGW(T)\ll\uPGW$ (so that $\RN$ is well-defined), and to obtain the bound on $\RN$ claimed in the statement of the lemma, we shall take the limit of depth $\ell\to\infty$. Let $T_\ell \equiv B_\ell(o;T)$; and let $\tree_\ell \equiv B_\ell(\vrt;\tree)$ where $\vrt$ is the root variable of $\tree$. Let $\mathscr{F}_\ell$ be the $\sigma$-field generated by $\tree_\ell$. Define the restricted measures
	\[
	\uPGW\equiv\uPGW\Big|_{\mathscr{F}_\ell}\,,\quad
	\uPGW_\ell(T)
	\equiv \uPGW(T)\Big|_{\mathscr{F}_\ell}
	= \uPGW_\ell(T_\ell)\,.
	\]
For finite $\ell$ it is clear that $\uPGW_\ell(T)\ll \uPGW_\ell$, so we can define the Radon--Nikodym derivative
	\[
	\RN_\ell \equiv
	\f{d\uPGW_\ell(T_\ell)}{d\uPGW_\ell}
	\equiv \RN_\ell(\tree_\ell ; T_\ell)\,,
	\]
which is a nonnegative martingale. We will show 
inductively that for all $\ell\ge0$ we have
	\beq\label{e:rn.induction.hypothesis}
	M(\ell)
	\equiv \max_{T_\ell} \E\bigg[
		\Big(\RN_\ell(\tree_\ell;T_\ell) \Big)^2
		\bigg]\le2\,,
	\eeq
where the maximum is taken over all depth-$\ell$ trees $T_\ell$ of maximum degree $\CC$. This will imply that the martingale $\RN_\ell$ is bounded in $L^2$,
so that (by the $L^2$ martingale convergence theorem)
it converges almost surely to a finite limit $\RN\equiv\RN_\infty$ satisfying the same $L^2$ bound.\smallskip

\noindent\bemph{Step 2. $L^2$ bound on martingale.}
We now prove \eqref{e:rn.induction.hypothesis}. The base case $\ell=0$ holds trivially since $\RN_0\equiv1$, so let us suppose inductively that the bound holds up to depth $\ell$. Suppose $T$ has root degree $t$, meaning that the variable $o$ has child clauses $a_1,\ldots,a_t$.
Suppose $a_i$ has $t(i)$ children in $T$, which we denote
$v_{i,1},\ldots,v_{i,t(i)}$. Let $\embed(T,\tree)$ denote the set of embeddings $T\hookrightarrow\tree$ that map root to root, and note that at depth one we have
	\[|\embed(T_1,\tree_1)|
	= (d)_t \prod_{i=1}^t (k-1)_{t(i)}
	\]
where $d$ is the root degree of $\tree$.
Let $\mathbf{E}$ denote expectation over $\tree$ and over
a uniformly random element $\zeta$ from $\embed(T_1,\tree_1)$.
 We then have the recursion
	\[\RN_{\ell+1}(\tree_{\ell_1};T_{\ell+1})
	= \f{\POIS_{\alpha k}(d-t)}{\POIS_{\alpha k}(d)}
	\mathbf{E}
	\bigg\{
	\prod_{ \substack{(i,j): 1\le i\le t,\\
		1\le j\le t(i) }}
	\RN_\ell\Big(\tree_{\ell}(\zeta(v_{ij})) ;T_{\ell+1}(v_{ij} ) \Big)
	\bigg\} \,,\]
for a rooted tree $U$ we write $U(x)$ for the subtree of $U$ descended from vertex $x$, and $U_\ell(x) \equiv B_{\ell-1}(x;U)$ (we view $U(x)$ and $U_\ell(x)$ as being rooted at $x$). We can then express the second moment as
	\[
	\E\Big[(\RN_{\ell+1})^2\Big]
	= \E\bigg[
	\bigg( \f{(d)_t}{(\alpha k)^t} \bigg)^2
	\mathbf{E}\bigg\{
	\prod_{ \substack{(i,j): 1\le i\le t,\\
		1\le j\le t(i) }}
	\RN_\ell\Big(\tree_{\ell}(\zeta(v_{ij})) ;T_{\ell+1}(v_{ij} ) \Big)
	\RN_\ell\Big(\tree_{\ell}(\xi(v_{ij})) ;T_{\ell+1}(v_{ij} ) \Big)
	\,\bigg|\, \mathscr{F}_1
	\bigg\} \bigg]
	\]
where $\mathbf{E}$ now refers to expectation over $\tree$ and over a uniformly random pair of elements $\zeta,\xi$ from $\embed(T_1,\tree_1)$. 
Let $\bm{j}(\zeta,\xi)$ count the number of pairs $(i,j)$ for which
$\zeta(v_{ij}) = \xi(v_{ij})$, and note that
	\[\mathbf{E}\bigg\{
	\prod_{ \substack{(i,j): 1\le i\le t,\\
		1\le j\le t(i) }}
	\RN_\ell\Big(\tree_{\ell}(\zeta(v_{ij})) ;T_{\ell+1}(v_{ij} ) \Big)
	\RN_\ell\Big(\tree_{\ell}(\xi(v_{ij})) ;T_{\ell+1}(v_{ij} ) \Big)
	\,\bigg|\, \mathscr{F}_1
	\bigg\}
	\le
	\mathbf{E}\bigg[ 
	M(\ell)^{\bm{j}(\zeta,\xi)}\bigg]\,,
	\]
for $M(\ell)$ as in \eqref{e:rn.induction.hypothesis}. Since $T$ has maximum degree $\CC$, we can rather crudely bound $\bm{i}(\zeta,\xi) \le \CC \bm{a}(\zeta,\xi)$ where $\bm{a}(\zeta,\xi)$ counts the number of indices $i$ for which
$\zeta(a_i) = \xi(a_i)$. For any fixed $\zeta$,
	\[
	\f{|\set{\xi : \bm{a}(\zeta,\xi) = s}|}
		{|\embed(T_1,\tree_1)|}
	=\f{ (d-t)_{t-s} (t)_s}{ (d)_t}\,.
	\]
It follows from this that
	\[
	\E\Big[(\RN_{\ell+1})^2\Big]
	\le
	\E\bigg\{
	\bigg( \f{(d)_t}{(\alpha k)^t} \bigg)^2
	\sum_{s=0}^t
	\f{ (d-t)_{t-s} (t)_s}{ (d)_t}
	\E\Big[(\RN_\ell)^2\Big]^{\CC s}
	\bigg\}
	=\f1{(\alpha k)^{2t}}\sum_{s=0}^t
	(t)_s\E[ (d)_{2t-s}]
	M(\ell)^{\CC s}\,.\]
Combining with the inductive hypothesis
\eqref{e:rn.induction.hypothesis} gives
	\[\E\Big[(\RN_{\ell+1})^2\Big]=\sum_{s=0}^t
	\f{(t)_s}{(\alpha k)^s}
	M(\ell)^{\CC s}
	\le \sum_{s\ge0}
	\bigg( \f{2\CC }{\alpha k} \bigg)^s \le 2\,,\]
where the last bound holds for $k\ge k_0$, and verifies the induction. As noted above, it follows from the $L^2$ martingale convergence theorem that $\uPGW(T)\ll \uPGW$, with Radon--Nikodym derivative $\RN$ satisfying the claimed bound.
\end{proof}
\end{lem}

\subsection{Orderliness and containment}\label{ss:orderly.contained}

As above, let $\tree$ denote a sample from the measure $\uPGW$. This subsection is primarily occupied with the proofs of the following two propositions:

\begin{ppn}\label{p:path.intersect.not.selfcont}
Let $\II\in\set{0,1}$, and let $\NotSC^\II\equiv\NotSC^\II(\tree)$ denote the set of all variables in $\tree$ that are not $\II$-self-contained. There is an absolute constant $k_0$ (depending only on 
the absolute constant $\DELTACONST$ which appears in Definition~\ref{d:contained}) such that for all $k\ge k_0$ we have
	\[
	\uPGW\left(
	\begin{array}{c}
	\text{$\tree$ contains a path $P$ of $\ell$ variables,}\\
	\text{emanating from $\vrt$,
	with $|\NotSC^\II\cap V_P|\ge\ell\ep$}
	\end{array}
	\right)
	\le 
	\f{\exp(k^2\ell)}
		{ \exp\{
		2^{k\DELTACONST/3} \ell\ep \} }\,,
	\]
for all $\ep\ge0$ and integer $\ell\ge0$.
\end{ppn}

\begin{ppn}\label{p:path.intersect.not.orderly}
Let $\II\in\set{0,1}$, and let $\NotOrd^\II\equiv\NotOrd^\II(\tree)$ denote the set of all variables in $\tree$ that are not $\II$-orderly. There is an absolute constant $k_0$ (depending only on the absolute constant $\DELTACONST$ which appears in Definition~\ref{d:orderly}) such that for all $k\ge k_0$ we have
	\[
	\uPGW\left(
	\begin{array}{c}
	\text{$\tree$ contains a path $P$ of $\ell$ variables,}\\
	\text{emanating from $\vrt$,
	with $|\NotOrd^\II\cap V_P|\ge\ell\ep$}
	\end{array}
	\right)
	\le 
	\f{\exp(k^2\ell)}
	{ \exp( 
	2^{k/4} \ell\ep/k^3 ) }\,,
	\]
for all $\ep\ge0$ and integer $\ell\ge0$.
\end{ppn}

From these propositions, it will be fairly straightforward to deduce the next two corollaries, which are the main consequences from the analysis in this subsection:

\begin{cor}\label{c:fair} For $\II\in\set{0,1}$, recall what it means for a variable to be $\II$-fair (Definition~\ref{d:perfect.fair}). For $\DELTACONST$ a sufficiently small positive absolute constant, we have for all $k\ge k_0$ (where $k_0$ is an absolute constant depending only on $\DELTACONST$) that
	\[
	\uPGW\Big(\text{$\vrt$ not $\II$-fair}\Big)
	\le \f1{ \exp( 2^{k\DELTACONST/3}R/k) }
	\]
for all $R$ as in \eqref{e:radii}, with $r\ge1$.
\end{cor}

\begin{cor}\label{c:excellent}
Recall what it means for a variable to be proper or improper (Definition~\ref{d:simple.types}). Further, for $\II\in\set{0,1}$, recall what it means for a variable to be $\II$-excellent (Definition~\ref{d:good.exc}). We have
	\[
	\uPGW\bigg(\text{$\vrt$ is improper or not $\II$-excellent}\bigg)
	\le \f1{ \exp( 2^{k\DELTACONST/4}R) }
	\]
for all $k\ge k_0$, where $k_0$ is an absolute constant (depending only on the absolute constant $\DELTACONST$).
\end{cor}

Proposition~\ref{p:path.intersect.not.orderly} is a straightforward consequence of the definitions together with Proposition~\ref{p:sparse.subtree.with.many.defects}, and we give its proof next. We then give the proof of Proposition~\ref{p:path.intersect.not.selfcont}, which is slightly more involved. Finally, at the end of this subsection we give the proofs for Corollaries~\ref{c:fair}~and~\ref{c:excellent}.

\begin{proof}[Proof of Proposition~\ref{p:path.intersect.not.orderly}]
If $\ell\ep=0$ the bound is vacuous, so we assume without loss that $\ell\ep\ge1$.
Suppose $\tree$ contains a path $P$ of $\ell$ variables emanating from $\vrt$, such that $S_0\equiv \NotOrd^\II\cap V_P$ has size $|S_0|\ge\ell\ep$. For any $u\in S_0$
there is a path $\gamma(u)\subseteq\tree$ that emanates from $u$, along which more than $(\DELTACONST)^3$ fraction of the variables are $\II$-defective. We will define a sequence $S_0\supseteq S_1 \supseteq \ldots$ as follows:
as long as $S_{s-1}\ne\emptyset$, take
	\[
	u_s
	\in\argmax\bigg\{ |\gamma(u)| : u\in S_{s-1}\bigg\}\,,
	\quad
	S_s
	\equiv S_{s-1} \,\bigg\backslash \,
		B_{2|\gamma(u_s)|+1}(u_s)\,.
	\]
Eventually this terminates at $S_{s_{\max}}=\emptyset$, and the resulting sequence of paths $\gamma(u_s)$ will be mutually disjoint. Since $S_0$ was a subset of the path $P$, we have
	\[
	\Big|S_{s-1}\setminus S_s\Big|
	\le \Big|V_P \cap B_{2|\gamma(u_s)|+1}(u_s)\Big|
	\le 4|V_{\gamma(u_s)}|\,.
	\]
Let $U$ be the union of the paths $\gamma(u_s)$: then
	\[
	|V_U|
	=\sum_{s=1}^{s_{\max}} |V_{\gamma(u_s)}|
	\ge
	\f14\sum_{s=1}^{s_{\max}} 
	\bigg( |S_{s-1}|-|S_s|\bigg)
	= \f{|S_0|}{4}
	\ge \f{\ell\ep}{4}\,.
	\]
Let $T = P \cup U$. Then $T\in\bLambda_{3,\ell'}(\tree)$ for some $\ell'\ge\ell$, and $| \DEFECTIVE^\II\cap V_T|\ge | \DEFECTIVE^\II\cap V_U|\ge |V_U|(\DELTACONST)^3$. Note that since $|V_U|\ge\ell\ep/4$
and also $|V_T|\le |V_P|+|V_U|$, we have
	\[
	|V_U|
	\ge \f{\ell\ep}{8}
		+\f{|V_T|-|V_P|}{2}
	= \f{\ell\ep}{8}
		+\f{|V_T|-|V_P|}{2}
	= \f{\ell\ep}{8}
		+\f{\ell'-\ell}{2}\,.
	\]
It follows that
	\begin{align*}
	&\uPGW\left(
	\hspace{-4pt}
	\begin{array}{c}
	\text{$\tree$ contains a path 
	$P$ of $\ell$ variables,}\\
	\text{emanating from $\vrt$,
	with $|\NotOrd^\II\cap V_P|\ge\ell\ep$}
	\end{array}
	\hspace{-4pt}
	\right)\\
	&\le
	\sum_{\ell'\ge\ell}
	\uPGW\Bigg(
	\textup{some $T\in\bLambda_{3,\ell'}(\tree)$
	has
	$|\DEFECTIVE^\II\cap V_T|\ge
	\bigg[
	 \f{\ell\ep}{8}
		+\f{\ell'-\ell}{2}
	\bigg](\DELTACONST)^3 $}
	\Bigg)\,.
	\end{align*}
It follows from Proposition~\ref{p:sparse.subtree.with.many.defects} that the last expression is upper bounded by
	\[
	\sum_{\ell'\ge\ell}
	\f{\exp(k^2[\ell + (\ell'-\ell)] )}
	{ \exp( 
		(\DELTACONST)^32^{k/4}
		[\ell\ep/8
		+ (\ell'-\ell)/2] /k^2 ) }
	\le 
	\f{O(1) \exp(k^2\ell)}
	{ \exp( (\DELTACONST)^3
	2^{k/4} \ell\ep /(8k^2) ) }\,,
	\]
and the result follows.
\end{proof}

We now turn to the proof of Proposition~\ref{p:path.intersect.not.selfcont}. For the reader's convenience, we repeat \eqref{e:def.R.for.containment} here:
	\[
	\mathfrak{R}^\II(v,t)
	\equiv	
	\sum_{u : t \le d(u,v) < 2\rprime}
	\f{
	\exp\{ k(\DELTACONST)^{-1} 
	\mathfrak{B}^\II(u,v) \} }
	{ \exp\{
		(k\log2)(1+\DELTACONST) d(u,v)
		\}}\,,
	\]
for $1\le t\le 2\rprime$. For any variable $v$, let
	\beq\label{e:orderly.Gamma}
	\Gm^\II(v)
	\equiv
	\left\{ \hspace{-3pt}
	\begin{array}{c} \text{paths $P$ emanating from $v$
		of length}\\\text{$|P|\le 2\rprime-1$
		with $|\DEFECTIVE^\II\cap V_P|
		\ge |V_P| (\DELTACONST)^3$}
		\end{array}\hspace{-3pt}\right\}\,.
	\eeq
(Recall that for a path $P$ of length $|P|=\ell-1$, the number of variables is $|V_P|=\ell$.) Let
	\[
	\gamma^\II(v)
	\equiv 1 + \max\bigg\{ |P| :
			P \in \Gm^\II(v) \bigg\}\,.
	\]
By definition, every path in $\Gm^\II(v)$ has length at most $2\rprime-1$. If $\Gm^\II(v)$ does not contain any (non-null) path, then we define $\gamma^\II(v)\equiv 1$. Therefore we always have $1\le \gamma^\II(v) \le 2\rprime$. If $\gamma^\II(v) \le d(u,v) \le 2\rprime-1$, then the shortest path between $u$ and $v$ must have less than $(\DELTACONST)^3$ fraction of $\II$-defective variables. Define also
	\beq\label{e:gamma.bar}
	\bar{\gamma}(v)
	\equiv\min\Bigg\{
	t\ge 1:
		|B_\ell(v)| \le 
		\exp\bigg\{ (k\log2)
			\bigg(1+\f{\DELTACONST}{2}\bigg)\ell \bigg\}
		\text{ for all }t \le \ell < 2\rprime
	\Bigg\}\,,\eeq
and note that $1 \le \bar{\gamma}(v) \le 2\rprime$. 
Recall the definition 
\eqref{eq-def-rad} of $\rad^\II(v)$.
We claim that
$\rad^\II(v)$ is upper bounded by
	\beq\label{e:def.chi.JJ}
	\chi^\II(v)
	\equiv\max\bigg\{
	\gamma^\II(v),
	\bar{\gamma}(v)
	\bigg\}\,,\eeq
where $1\le \chi^\II(v) \le 2\rprime$. 
Indeed, if $t\ge \chi^\II(v)$, then (by taking $k\ge k_0$ with $k_0$ an absolute constant) we have 
	\begin{align*}
	\mathfrak{R}^\II(v,t)
	&\le
	\sum_{\ell=t}^{2\rprime-1}
	\f{|B_\ell(v)|
	\exp\{
	k(\DELTACONST)^{-1} (\DELTACONST)^3
		(1+\ell)\}}{ \exp\{ (k\log 2) (1+\DELTACONST) \ell\} }\\
	&\le
	\sum_{\ell=t}^{2\rprime-1}
	\f{\exp\{k(\DELTACONST)^2(1+\ell)\}}
	{\exp\{(k\log 2) (\DELTACONST/2) \ell\} }
	\le
	\f{ O(1) \exp\{ 2k(\DELTACONST)^2 \}}
		{ \exp\{ (k\log2) (\DELTACONST/2) \} }
	\le
	\f1{\exp(k\DELTACONST/3)}
	\le\f14\,.
	\end{align*}
This shows that $\rad^\II(v) \le \chi^\II(v)$, and we now turn to controlling $\chi^\II(v)$. The sketch for the proof of Proposition~\ref{p:path.intersect.not.selfcont} is as follows: in Lemma~\ref{l:thick.sparse.subtree} we show that if $\tree\sim\uPGW$,
then it is very unlikely to have a 
sparse subtree $T\subseteq\tree$ where many variables have a large value of $\bar{\gamma}(u)$. In Lemma~\ref{l:non-self-contained} we show that if $\tree$ has a path $P$ emanating from $\vrt$ with many variables that are not $\II$-self-contained, then there is a sparse subtree $P\subseteq T\subseteq\tree$ such that either 
(i) $T$ has many variables with a large value of $\bar{\gamma}(u)$, or
(ii) $T$ has many $\II$-defective variables. The probability of case~(i) is bounded by Lemma~\ref{l:thick.sparse.subtree} while that of case~(ii) is bounded by Proposition~\ref{p:sparse.subtree.with.many.defects}, and the result of Proposition~\ref{p:path.intersect.not.selfcont} follows. 
The remainder of this subsection gives the details of this argument.

\begin{dfn}\label{d:thick}
For $T\in \bLambda_{\CC,\ell}(\tree)$ and $S\subseteq V_T$,
we say that the pair $(T,S)$ is \bemph{$\ep$-thick}
(in $\tree$) if the neighborhoods $B_{\bar{\gm}(u)}(u)$ for $u\in S$ are mutually disjoint with
	\[	\sum_{u\in S}
	\Big(\bar{\gm}(u) -1\Big) 
	\ge \ell\ep\,.
	\]
(Note that variables $u$ with $\bar{\gamma}(u)=1$ give no contribution to the above sum.) We then say that $T$ is \bemph{$\ep$-thick} if $(T,S)$ is \bemph{$\ep$-thick} for some $S\subseteq V_T$.
\end{dfn}

\begin{lem}\label{l:thick.sparse.subtree}
For $k\ge k_0$ (an absolute constant depending only on the absolute constant $\DELTACONST$), we have
		\[
	\uPGW\bigg(
	\text{some $T\in\bLambda_{3,\ell}(\tree)$
	is $\ep$-thick in $\tree$}\bigg)
	\le \f{\exp(k^2\ell) }
		{ \exp\{
		2^{ k\DELTACONST/2}
	\ell\ep/k^5
		\} }
	\]
for all $\ep\ge0$ and integer $\ell\ge0$.

\begin{proof}
We shall bound the expected number of $\ep$-thick subtrees of $\tree$. To this end, first let us fix a sparse tree $T\in\bLambda_{3,\ell}$ and a subset of variables $S\subseteq V_T$. Similarly as in the proof of Proposition~\ref{p:sparse.subtree.with.many.defects}
(cf.\ \eqref{e:expected.number.of.embeddings}~and~\eqref{e:embeddings.times.embedded.probab}), we 
use Markov's inequality to bound
	\begin{align}\nonumber
	&\uPGW\bigg(
	\textup{there exists an
	 embedding $\zeta:T\hookrightarrow\tree$
	such that $(\zeta(T),\zeta(S))$
	 is $\ep$-thick in $\tree$}
	\bigg) \\ \nonumber
	&\le
	\int\bigg|\bigg\{
	\textup{embeddings
	$\zeta:T\hookrightarrow\tree$
	such that $(\zeta(T),\zeta(S))$
	 is $\ep$-thick in $\tree$}
	\bigg\}\bigg| \,d\uPGW(\tree)\\
	&\le (\alpha k^2)^{3\ell}
	\cdot \Big[\uPGW(T) \Big]\bigg(
		\textup{$(T,S)$ is $\ep$-thick
		in $\tree$}\bigg)\,,
	\label{e:expected.number.of.thick.subtrees}
	\end{align}
where in the last line $(\alpha k^2)^{3\ell}$ is an upper bound on $\embed(T)$ (similar to \eqref{e:lebesgue.vol.of.embeddings}). Now let $\lambda$ denote any tuple of integers $(\lambda(u):u\in S)$ such that $\lambda(u)\ge 1$ for all $u\in S$, and 
	\[|\lambda|
	\equiv \sum_{u\in S}\lambda(u)
	\ge \ell\ep\,.
	\]
Recalling the definition \eqref{e:gamma.bar} of $\bar{\gamma}$, we see that if $\bar{\gamma}(u) -1 \ge\lambda(u)\ge1$ then the ball $|B_{\lambda(u)}(u)|\subseteq\tree$ must be large. Therefore, writing $T_u$ for the tree $T \cap B_{\lambda(u)}(u)$ 
rerooted at $u$, we have
	\begin{align*}
	&\Big[\uPGW(T) \Big]\Bigg(
		\textup{$(T,S)$ is $\ep$-thick
		in $\tree$}\Bigg)\\
	&\le
	\sum_\lambda
	\Big[\uPGW(T) \Big]
	\Bigg(
	\hspace{-3pt}\begin{array}{c}
	\textup{$|B_{\lambda(u)}(u)|
		\ge \exp\{(k\log2) (1+\DELTACONST/2) \lambda(u) \}$}\\
	\textup{for all $u\in S$, and the $B_{\lambda(u)}(u)$
	 are mutually disjoint}
	\end{array}\hspace{-3pt}
	\Bigg)\\
	&\le
	\sum_\lambda
	\prod_{u\in S}
		\Big[ \uPGW(T_u)\Big]
	\Bigg( 
	\Big|B_{\lambda(u)}(\vrt)\Big|
	\ge \exp\bigg\{
		(k\log2) \bigg(1+\f{\DELTACONST}2\bigg)
		\lambda(u)\bigg\}
	 \Bigg)\\
	&\le
	\sum_\lambda
	\prod_{u\in S}\Bigg\{ 
	2 \bigg( \f{e}
		{ \exp\{ (2^{k\DELTACONST/2}/k^3)^{\lambda(u)}
			 \} }\bigg)^{1/2}
			 \Bigg\}\,,
	\end{align*}
where the last inequality follows by
Lemma~\ref{l:martingale.bound.PGW}, Lemma~\ref{l:radon.derivative.pgwT},
and the Cauchy--Schwarz inequality. Note that for all $A\ge e$ and $x\ge1$ we have $\log x \le x-1 \le (\log A)(x-1)$,
and rearranging gives $A^x\ge Ax$. It follows that
	\[\Big[\uPGW(T) \Big]\bigg(
		\textup{$(T,S)$ is $\ep$-thick
		in $\tree$}\bigg)
	\le\sum_\lambda
	\f{\exp\{O(1)|S|\}}{ \exp\{
	(2^{k\DELTACONST/2}/k^3) |\lambda| / 2 
	\} }\]
(where, as always, $O(1)$ refers to an absolute constant). Note that $|S|\le\ell$,
and the number of distinct tuples $\lambda$
with $|\lambda|=L$ is crudely upper bounded by
	\[\binom{L-1}{|S|-1} \le 2^L\,.\]
Substituting these bounds into the preceding calculation gives
	\[\Big[\uPGW(T) \Big]\bigg(
		\textup{$(T,S)$ is $\ep$-thick
		in $\tree$}\bigg)
		\le 
	\sum_{L\ge \ell\ep}
	\f{2^L \exp( O(\ell))}{
	\exp\{ (2^{k\DELTACONST/2}/k^3) L/2 \}
	}
	\le \f{\exp(O(\ell))}
		{\exp\{2^{k\DELTACONST/2}\ell\ep /k^4
		 \} }\,.
	\]
The result follows by substituting this into \eqref{e:expected.number.of.thick.subtrees},
and then summing over all $(T,S)$. (The size of $\bLambda_{3,\ell}$ is bounded by $\exp(O(\ell))$;
and for any given $T\in\bLambda_{3,\ell}$,
the number of subsets $S\subseteq V_T$ is clearly at most $2^\ell$.)
\end{proof}
\end{lem}

\begin{lem}\label{l:non-self-contained}
Let $\tree$ be any bipartite factor tree rooted at variable $\vrt$. Let $\NotSC^\II\equiv \NotSC^\II(\tree)$ denote the subset of variables in $\tree$ that are not $\II$-self-contained. For $\ep>0$ and $\ell\ge1$,
if $\tree$ contains a path $P$ emanating from $\vrt$ with $|V_P|=\ell$ and $|\NotSC^\II\cap V_P|\ge\ep\ell$, then at least one of the following must occur:
\begin{enumerate}[(i)]
\item \label{i:case.thick}
There is a tree $T\in\bLambda_{3,\ell'}(\tree)$
(with $\ell'\ge\ell$) which is $\ep'$-thick for $\ep'$ defined by
	\[
	\ell'\ep'
	=\f1{4} \bigg[
	(\ell'-\ell)
	+ \f{\ell\ep}{9}\bigg]\,.
	\]
\item \label{i:case.defect}
There is a tree $T\in\bLambda_{4,\ell'}(\tree)$
(with $\ell'\ge\ell$) with
	\[|\DEFECTIVE^\II \cap V_T| \ge 
	\f{(\DELTACONST)^3}{4}
	\bigg[\f{\ell'-\ell}{2}
	+\f{\ell\ep}{9}
	\bigg]
	\,.\]
\end{enumerate}

\begin{proof}
Let $S_0\equiv \NotSC^\II \cap V_P$. For each $u\in S_0$, by definition there exists $g(u)$ such that 
	{\setlength{\jot}{0pt}\begin{align}
	\nonumber
	&1\le d(u,g(u)) \le \rprime\,,\\
	&1\le d(u,g(u)) < \rad^\II(g(u))
		\le \chi^\II(g(u)) \le 2\rprime\,.
	\label{e:chi.at.least.two}
	\end{align}}%
We will first build a tree $T'$,
with $P\subseteq T' \subseteq\tree$,
by combining $P$ with paths between $u$ and $g(u)$ for a subset of $u$ from $S_0$, as follows.
We will keep track of 
	{\setlength{\jot}{0pt}\begin{align*}
	\NotSC^\II \cap V_P
	&=S_0 \supseteq S_1 
	\supseteq\ldots \supseteq S_{s_{\max}}=\emptyset\,,\\
	P
	&=T_0\subseteq T_1 \subseteq \ldots \subseteq
	T_{s_{\max}} = T'\,.
	\end{align*}}%
As long as $S_{s-1}\ne\emptyset$, we can choose $u_s$ from $S_{s-1}$ such that $g(u_s)$ achieves
	\[
	\chi^\II(g(u_s))
	= \chi_s
	\equiv\max\bigg\{
	\chi^\II(g(u))
	: u\in S_{s-1}
	\bigg\}\,.
	\]
Let $P_s$ be the path joining $u_s$ to $g(u_s)$; it has length $|P_s|= d(u_s,g(u_s)) \le \rad^\II(g(u_s))-1 \le \chi_s-1$. Let
	{\setlength{\jot}{0pt}\begin{align*}
	T_s &\equiv T_{s-1} \cup P_s\,,\\
	S_s &\equiv S_{s-1} \setminus 
	B_{4(\chi_s-1)}(g(u_s))\,.
	\end{align*}}%
Then the paths $P_s$ will be mutually disjoint: indeed, for any $s<t$, by definition the variable $u_t$ must lie outside the ball $B_{4(\chi_s-1)}(g(u_s))$,
so the distance between $g(u_s)$ and $u_t$ 
must be at least $4(\chi_s-1)+1$.
Since $g(u_s)\in P_s$
and $u_t\in P_t$, we conclude that the minimum distance between $P_s$ and $P_t$ must satisfy
	\[
	d(P_s,P_t)
	\ge d(g(u_s),u_t)
		-|P_s|-|P_t|
	\ge
	\bigg(4(\chi_s-1)+1\bigg)
	- (\chi_s-1)-(\chi_t-1)
	\ge 2\chi_s-1\,,
	\]
where the last step uses that $\chi_s\ge\chi_t$. We then continue the procedure until we reach $S_{s_{\max}}=\emptyset$, at which point we define $T'\equiv T_{s_{\max}}$. Note
	\beq\label{e:candidate.thick.tree}
	|V_{T'}|
	\le \ell + \sum_{s=1}^{s_{\max}} d(u_s,g(u_s))
	\le \ell + \sum_{s=1}^{s_{\max}}
		(\chi_s-1)\,.
	\eeq
The intersection between $B_{4(\chi_s-1)}(g(u_s))$ with (any) path $P$ can contain at most $8(\chi_s-1)+1 \le 9(\chi_s-1)$ variables, where the last inequality holds since $\chi_s\ge2$ from \eqref{e:chi.at.least.two}. Therefore $|S_s| \ge |S_{s-1}|-9(\chi_s-1)$, which implies
	\beq\label{e:large.sum.chi}
	\sum_{s=1}^{s_{\max}} (\chi_s -1)
	\ge 
	\sum_{s=1}^{s_{\max}} \f{|S_{s-1}|-|S_s|}{9}
	= \f{|S_0|}{9} \ge \f{\ell\ep}{9}\,.
	\eeq
Recall \eqref{e:def.chi.JJ} that $\chi^\II(v)=\max\set{ \gamma^\II(v),\bar{\gamma}(v)}$, so one of the following must hold:
	\begin{align}
	\label{e:i.case.tree.with.many.defects}
	\sum_{s=1}^{s_{\max}} 
	\Big(\gamma^\II( g(u_s)) -1\Big)
	&\ge \f12 \sum_{s=1}^{s_{\max}}
	\Big(\chi^\II( g(u_s)) -1\Big)\,,\\
	\label{e:ii.case.thick.tree}
	\sum_{s=1}^{s_{\max}} 
	\Big(\bar{\gamma}( g(u_s)) -1\Big)
	&\ge \f12 \sum_{s=1}^{s_{\max}}
	\Big(\chi^\II( g(u_s)) -1\Big)
	\end{align}
We consider separately the two cases:
\begin{enumerate}[(i)]
\item[\eqref{i:case.thick}] If \eqref{e:ii.case.thick.tree} occurs, then we let $T=T'$. Combining with
\eqref{e:candidate.thick.tree}
and \eqref{e:large.sum.chi} gives
	\[
	\sum_{s=1}^{s_{\max}} \Big(\bar{\gamma}(g(u_s)) -1 \Big)
	\ge
	\f12\sum_{s=1}^{s_{\max}} (\chi_s-1)
	\ge
	\f14\bigg[
	 (\ell'-\ell)
	+ \f{\ell\ep}{9}
	\bigg]
	\,.
	\]
Therefore the tree $T=T'$ is $\ep'$-thick
for $\ep'$ as defined in the statement of the lemma.

\item[\eqref{i:case.defect}] If \eqref{e:i.case.tree.with.many.defects} occurs,
then for each $s$ we can find a path $Q_s$ emanating from $g(u_s)$ with $\gamma^\II(g(u_s))$ variables,
of which at least $(\DELTACONST)^3$ fraction are $\II$-defective. Let $T$ be the union of $T'$ with all the paths $Q_s$, for $1\le s\le s_{\max}$. The $Q_s$ are mutually disjoint. Recalling \eqref{e:candidate.thick.tree}, we have
	\beq\label{e:size.bound.replacing.thick.tree.eqn}
	\ell'=|V_T|\le
	\ell + \sum_{s=1}^{s_{\max}} (\chi_s-1)
	+ \sum_{s=1}^{s_{\max}} \Big(\gamma^\II(g(u_s)) -1\Big)
	\le 
	\ell +2 \sum_{s=1}^{s_{\max}} (\chi_s-1)\,.\eeq
Combining 
\eqref{e:large.sum.chi},
\eqref{e:i.case.tree.with.many.defects},
and \eqref{e:size.bound.replacing.thick.tree.eqn}
 gives
	\begin{align*}
	|\DEFECTIVE^\II \cap V_T|
	&\ge \sum_{s=1}^{s_{\max}}
	|\DEFECTIVE^\II \cap Q_s|
	\ge (\DELTACONST)^3
	\sum_{s=1}^{s_{\max}} \gamma^\II(g(u_s))
	\ge \f{(\DELTACONST)^3}{2}
	\sum_{s=1}^{s_{\max}} (\chi_s-1)\\
	&\ge
	\f{(\DELTACONST)^3}{4}
	\bigg[\f{\ell'-\ell}{2}
	+\f{\ell\ep}{9}
	\bigg]\,.
	\end{align*}
\end{enumerate}
Combining the two cases gives the claim.
\end{proof}
\end{lem}

\begin{proof}[Proof of Proposition~\ref{p:path.intersect.not.selfcont}]
It follows from Lemma~\ref{l:non-self-contained} that
	\[
	\uPGW\left(
	\begin{array}{c}
	\text{$\tree$ contains a path 
	$P$ of $\ell$ variables,}\\
	\text{emanating from $\vrt$,
	with $|\NotSC^\II\cap V_P|\ge\ell\ep$}
	\end{array}
	\right)
	\le P_\textup{(i)} + P_{(ii)}
	\]
where $P_{\eqref{i:case.thick}}$ and $P_{\eqref{i:case.defect}}$ refer to the two cases of
Lemma~\ref{l:non-self-contained}. 
For case~\eqref{i:case.thick}, we use Lemma~\ref{l:thick.sparse.subtree} to bound
	\[P_{\eqref{i:case.thick}}
	\le
	\sum_{\ell'\ge\ell}
	\f{\exp\{k^2[\ell + (\ell'-\ell)]\}}
	{
	\exp\{ 2^{k\DELTACONST/2}
	[\ell\ep/36
		+ (\ell'-\ell)/4]
		/k^5\}
	}
	\le \f{\exp(k^2\ell)}
		{ \exp\{
		2^{k\DELTACONST/2} \ell\ep / k^6\} }\,.
	\]
For case~\eqref{i:case.defect}, we use
Proposition~\ref{p:sparse.subtree.with.many.defects} to bound
	\[
	P_{\eqref{i:case.defect}}
	\le
	\sum_{\ell'\ge\ell}
	\f{\exp\{k^2[\ell + (\ell'-\ell)]\}}
	{\exp\{(\DELTACONST)^32^{k/4}
	[\ell\ep/36 + (\ell'-\ell)/8]
	/k^2\}}
	\le 
	\f{\exp(k^2\ell)}
		{ \exp\{
		2^{k/4} \ell\ep/k^3 \} }\,.
	\]
Combining gives the result.
\end{proof}

We conclude this subsection with the proofs of Corollaries~\ref{c:fair}~and~\ref{c:excellent}, which are easy consequences of what was proved above.

\begin{proof}[Proof of Corollary~\ref{c:fair}]
Recall from Definition~\ref{d:perfect.fair} that an acyclic variable is termed $\II$-fair if (i) it is $\II$-stable; (ii) its $5\rprime$-neighborhood contains no more than $\exp\{k^2(5\rprime)\}$ variables; and (iii) every length-$\rprime$ path emanating from it contains at least one $\II$-perfect variable, where $\II$-perfect means both $\II$-orderly and $\II$-self-contained. For condition (i), it follows from Proposition~\ref{p:one.stable} that
	\[
	\uPGW\Big(\text{$\vrt$ not $\II$-stable}\Big)
	\le \f1{\exp(2^{k/20}R)}\,.
	\]
For condition (ii), it follows from Lemma~\ref{l:martingale.bound.PGW}
and Markov's inequality that
	\[
	\uPGW\bigg(
	|B_{5\rprime}(\vrt)| \ge \exp\{k^2(5\rprime)\}
	\bigg)
	\le\f{ e }
	{	\exp \{ \exp\{k^2(5\rprime) \}
	 / (\alpha k^2)^{5\rprime}\}}
	\le \f1{\exp(\exp(kR))}\,.
	\]
For condition (iii), it follows by combining Propositions~\ref{p:path.intersect.not.selfcont}~and~\ref{p:path.intersect.not.orderly} that
	\[
	\uPGW\left(
	\begin{array}{c}
	\text{$\tree$ contains a path $P$ 
	of $\rprime$ variables,}\\
	\text{emanating from $\vrt$,
	with $V_P\subseteq\NotSC^\II\cup\NotOrd^\II$}
	\end{array}
	\right)
	\le 
	\f{O(1) \exp(k^2 \rprime)}
	{ \exp( \Omega( 2^{k\DELTACONST/3} \rprime)) }\,.
	\]
Combining these bounds gives the result.
\end{proof}

\begin{proof}[Proof of Corollary~\ref{c:excellent}]
Recall from Definition~\ref{d:good.exc} that an acyclic variable $v$ is termed $\II$-excellent if its neighborhood $T=B_{10\rprime}(v)$ satisfies condition~\eqref{e:j.excellent}, which we repeat here for convenience:
	\[
	p_{\textup{ex},\II}(T)\equiv
	\uPGW\left( 
	\left.\begin{array}{c}
	B_{20\rprime}(u)
	\text{ contains 
	any}\\
	\text{variable which is not $\II$-fair}
	\end{array}
	\, \right| \, B_{10\rprime}(u) \cong T
	\right) \le \f1{\exp\{k^3 \rprime \}}\,.
	\]
By Markov's inequality and iterated expectations,
	\begin{align*}
	p
	&\equiv
	\uPGW\Big(\text{$\vrt$ not $\II$-excellent}\Big)
	\le
	\int
	\f{p_{\textup{ex},\II}(
		B_{10\rprime}(\vrt;\tree) )}
		{1/\exp\{k^3 \rprime \}}
	\,d\uPGW(\tree)\\
	&=\exp\{k^3 \rprime \}
	\,\uPGW\bigg(
	\textup{$B_{20\rprime}(\vrt;\tree)$ contains any variable which is not $\II$-fair}
	\bigg)\,.
	\end{align*}
By another application of Markov's inequality,
together with the unimodularity property
\eqref{e:unimodular}, we find
	\begin{align*}
	p &\le\exp\{k^3 \rprime \}
	\int
	\sum_{u\in B_{20\rprime}(\vrt;\tree)}
	\mathbf{1}\bigg\{\textup{$u$ is not $\II$-fair}\bigg\}
	\,d\uPGW(\tree)\\
	&=\exp\{k^3 \rprime \}
	\int
	|B_{20\rprime}(\vrt;\tree)|
	\mathbf{1}
	\Big\{\textup{$\vrt$ not $\II$-fair}\Big\}
	\,d\uPGW(\tree)
	\le
	p_\textup{small} + p_\textup{large}
	\,,
	\end{align*}
where $p_\textup{small}$ is the contribution from
the event
$|B_{20\rprime}(\vrt;\tree)|
	\le \exp\{k^2R\}$,
	and $p_\textup{large}$
is the contribution from
the complementary event. It follows from Corollary~\ref{c:fair} that
	\[
	p_\textup{small}
	=\exp\{ k^3 \rprime+ k^2 R\}
	\uPGW\Big(\textup{$\vrt$ not $\II$-fair}\Big)
	\le \f1{\exp( 2^{k\DELTACONST/3}R/k^2) }\,.
	\]
The contribution from the complementary event is
	\begin{align*}
	\f{p_\textup{large}}{\exp\{k^3 \rprime\}}
	&=
	\int 
	|B_{20\rprime}(\vrt;\tree)|
	\mathbf{1}
	\bigg\{
	|B_{20\rprime}(\vrt;\tree)|
	\ge \exp\{k^2R\}
	\bigg\}\,d\uPGW(\tree)\\
	&\le
	\exp\{k^2R\}
	\uPGW\bigg(|B_{20\rprime}(\vrt;\tree)|
	\ge \exp\{k^2R\}\bigg)
	+\int_{\exp\{k^2R\}}^\infty
		\uPGW\bigg(|B_{20\rprime}(\vrt;\tree)|
	\ge s\bigg)\,ds\,.
	\end{align*}
It follows from Lemma~\ref{l:martingale.bound.PGW} that this is very small: $p_\textup{large} \le 1/\exp(\exp(kR))$, which is negligible compared with the bound on $p_\textup{small}$. Combining these gives the claimed bound (it is easy to show, using Markov's inequality, that the probability for $\vrt$ to be improper is negligible).
\end{proof}

\subsection{Combinatorial analysis of preprocessing}
\label{ss:preprocessing.structural}

In this section we analyze the preprocessing algorithm described by Definition~\ref{d:proc}, which maps the original $\ksat$ instance $\GG$ to its pruned version $\proc\GG$. Recall that the procedure starts from an initial set $A\subseteq V$, which is the set of all variables in $\GG$ that are improper (Definition~\ref{d:simple.types}) or not $1$-good (Definition~\ref{d:good.exc}). It then iteratively produces a sequence $\GG\supseteq{}_0\GG_A \supseteq{}_1\GG_A\supseteq\ldots$, terminating in $\proc\GG=\GG\setminus\bsp'(A;\GG)$. It is a straightforward consequence of Corollary~\ref{c:excellent} that the fraction of variables in the initial set $A$ is $o_R(1)$, so the main challenge is to bound the effect of the $\bsp'$ procedure. We do this in two parts:
\begin{enumerate}[a.]
\item The current subsection (\S\ref{ss:preprocessing.structural}) is devoted to the proof of a structural result, Proposition~\ref{p:corrupted.tree.in.graph},
which says roughly that if $\bsp'(A;\GG)$ has a large connected component, then $\GG$ must contain a certain kind of subgraph
(either a ``bicycle'' or a ``sparse corrupted subtree'') of comparable size. This statement holds deterministically.
\item In the next subsection (\S\ref{ss:preprocessing.probabilistic}) we bound the probability for such subgraphs to occur
(Lemma~\ref{l:corrupt.tree.in.erdos.renyi}). This allows us to control the typical size of $\bsp'(A;\GG)$, and we conclude
\S\ref{ss:preprocessing.probabilistic} with the proof of Proposition~\ref{p:small.fraction.removed.in.processing}.
\end{enumerate}
Turning to the task of this subsection, we note that $\bsp'$ is highly analogous to the $\bsp$ procedure \eqref{eq-bootstrap-percolation}, which determined the $\II$-defective variables based on the initial set $D^{\KAPPA,\II}$ (see Definition~\ref{d:j.defective}). Recall from \S\ref{ss:bootstrap.defects} the result Lemma~\ref{l:larger.tree.intersecting.bsp.init.set}, which roughly says that if $\tree$ has a sparse subtree $T$ that has a large intersection with the $\II$-defective set, then it must also have a sparse subtree $T'$ ($T\subseteq T'\subseteq\tree$) that has a large intersection with the initial set $D^{\KAPPA,\II}$. We now prove an analogous result for $\bsp'$, given by Proposition~\ref{p:corrupted.tree.in.graph} below. To state the result, we introduce the following definition:

\begin{dfn}\label{d:corrupt} Let $\GG\equiv(V,F,E)$ be a bipartite factor graph, and fix any $A\subseteq V$. We say that a subtree $T\subseteq\GG$ is \bemph{$\ep$-corrupted with respect to $(A,\GG)$} if there is a subset of vertices $B\subseteq V_T\cap A$ such that $|B|\ge\ep|V_T|$, the minimum pairwise distance between vertices in $B$ exceeds $2R$, and the graph \[T \cup B_R(B;\GG)\] is acyclic. (We may say simply ``$\ep$-corrupted'' if $(A,\GG)$ is unambiguous.)\end{dfn}

\begin{ppn}\label{p:corrupted.tree.in.graph} Let $\GG=(V,F,E)$ be a finite bipartite factor graph in which all clauses have degree $k$, and let $A$ be any subset of $V$. If $\bsp'(A;\GG)$ contains a connected component $\mathscr{B}$ of diameter at least $5L$ where $L\ge 400R$, then either (i) there is a subgraph $B'\subseteq\mathscr{B}$ with $\diam B' \le 11L$ that contains at least two cycles; or (ii) there is an $1/(80R)$-corrupted subtree $T\subseteq\mathscr{B}$ of maximum degree at most four with $|V_T|\ge L$ and $L \le \diam T \le 11L$.
\end{ppn}

The remainder of this subsection is devoted to the proof of Proposition~\ref{p:corrupted.tree.in.graph}.

\begin{lem}\label{l:marked.blocks.corrupt.tree}Let $\tree=(V,F,E)$ be a finite bipartite factor tree in which all clauses have degree $k$. Assume that $\tree$ has diameter $L\ge 8R$. If $A$ is a subset of $V$ such that $\bsp'(A;\tree)=\tree$, then there is a $1/(80R)$-corrupted subtree $T\subseteq\tree$ of diameter $L$ and maximum degree at most four.

\begin{proof} Since $\diam\tree = L$, it must contain some path $P$ of length $L$. Let $\vrt$ denote one of the endpoints of this path. From now on we regard $\tree$ as being rooted at $\vrt$.

\smallskip

\noindent\bemph{Step 1. Simplified bootstrap percolation of marked blocks.} We now define a much simpler bootstrap percolation process, which we will show (in subsequent steps) to dominate $\bsp'(A;\tree)$ in an appropriate sense. For integers $j\ge0$, let $\tree[j]$ denote the subgraph of $\tree$ induced by variables whose distance to the root $\vrt$ lies between $4Rj$ and $4R(j+1)$; we then consider each connected component of $\tree[j]$ as a ``block.'' Note that the blocks themselves have a tree-like (i.e., hierarchical) structure: if $U$ is a block rooted at depth $4R(j+1)$, then its root is a leaf of a block $U'$ rooted at depth $4Rj$, and we say that $U$ is a ``child block'' of $U'$. Let $\blocks$ be the tree structure of blocks. Let $\ablock_0$ be the set of all blocks that intersect $A$. For $t\ge1$ we will say that a block belongs to $\ablock_t$ if either it belongs to $\ablock_{t-1}$, or has at least two child blocks in $\ablock_{t-1}$. Iterate this to define the set of all ``marked'' blocks,
	\[\MARK(\ablock_0;\mathfrak{B})
	\equiv\bigcup_{t\ge0}\ablock_t\,.\]
Note that this ``marking'' is also a bootstrap-percolation-type process, but is much simpler than $\bsp'$.\smallskip

\noindent\bemph{Step 2. A general property of $\bsp'(A;\tree)$.} We next state and prove a useful property of the $\bsp'(A;\tree)$ process: if $S$ is any subset of $V$ which does not intersect $A$, then in the first round of $\bsp'(A;\tree)$ in which \emph{any} variable is deleted from $S$, it must be the case that some variable on its internal boundary
	\[\partial^\textup{int} S
	\equiv\bigg\{
	u\in S : \textup{$u\in B_1(w;\tree)$
		for some $w\in V\setminus S$}
	\bigg\}\]
is also deleted. To see that this property holds, say the first deletion from $S$ occurs at round $t+1$. This means that after round $t$, the subgraph remaining in the $\bsp'(A;\tree)$ process is $\tree_t=(V_t,F_t,E_t)\subseteq\tree$ with $V_t\supseteq S$ (i.e., no variable has yet been deleted from $S$). The subgraph remaining after the next round $t+1$ is $\tree_{t+1}=(V_{t+1},F_{t+1},E_{t+1})\subseteq\tree_t$ where $V_t\setminus V_{t+1}$ intersects $S$. 
By the definition of $\bsp'$,
	\[\tree_{t+1} = \tree_t
	\,\bigg\backslash \, 
	B_R(A_t;\tree_t)\]
where $A_{-1}=A$ and $A_t=\ACT(\tree_t)$ for $t\ge1$. Now suppose for contradiction that $V_t\setminus V_{t+1}$ does not intersect $\partial^\textup{int} S$. For any variable $v$, its neighborhood $B_R(v;\tree_t)$ is a connected subgraph of $\tree_t$ (hence also of $\tree$). As a result it must be that for all $v\in A_t$, the neighborhood $B_R(v;\tree_t)$ is either disjoint from $S$, or contained in $S\setminus \partial^\textup{int}S$. Since $V_t\setminus V_{t+1}$ intersects $S$, there must be at least one $v\in A_t$ with $B_R(v;\tree_t)\subseteq S\setminus \partial^\textup{int} S$. Since $S$ does not intersect $A$, it must be that $t\ge0$ and $v\in A_t=\ACT(\tree_t)$. By definition of $\ACT(\tree_t)$,
this means that some clause $a$ in $B_R(v;\tree_t)$ must have degree less than $k$. Since all clauses in $\tree$ were assumed to be of degree $k$, this means that in $\tree$ there was a variable $u\in \pd a$ which was deleted by the end of round $t$. However $u$ must also belong to $S$, contradicting the hypothesis that no deletion occurred from $S$ by the end of round $t$. This proves the claim.\medskip

\noindent\bemph{Step 3. Comparison of marking and $\bsp'$.} We now argue that the marking process ``dominates'' $\bsp'(A;\tree)$ in the following sense. Let us say that a block $U$ is a ``tall block'' if it has depth at least $2R$; otherwise we call $U$ a ``shallow block.'' We claim that any tall block must be marked. Suppose for contradiction that this is not the case, then let $U$ be any block of maximal depth among all the unmarked tall blocks. This means that all the child blocks of $U$ must be marked, or be shallow blocks. Since $U$ itself is unmarked, the definition of the marking process implies that at most one child block of $U$ can be marked. If such a block exists we will denote it $U'$. We let $W$ denote any child block of $U$ which is unmarked; this means that $W$ must be a shallow block.

We now argue that throughout the $\bsp'(A;\tree)$ process, $W$ has no influence on $U$. To see this, apply the claim from Step 2 with $S=W$. Since $W$ is a shallow block, it follows that at the first round during $\bsp'(A;\tree)$ in which any variable is removed from the upper half of $W$, the root of $W$ must also be removed. After this, any variables remaining in $W$ will be disconnected from $U$. Thus, throughout $\bsp'(A;\tree)$, there is no time at which any variable $u\in U$ has in its $(3R/10)$-neighborhood a clause $a$ of degree less than $k$ that lies in $W$. This shows that $W$ has no influence on the evolution of $U$ under $\bsp'(A;\tree)$.

Let $\rho$ denote the root of $U$, and let $\rho'$ denote the root of $U'$ (if it exists). Let $\tau$ be the first time during $\bsp'(A;\tree)$ that any variable is deleted from $U$; note that $\tau$ must be finite by the assumption that $U$ is a tall block. Let $\tau_\rho$ be the time that $\rho$ is deleted, and let $\tau_{\rho'}$ be the time that $\rho'$ is deleted. It follows from the preceding discussion that $\tau=\min\set{\tau_\rho,\tau_{\rho'}}$. Let us suppose first $\tau=\tau_\rho<\tau_{\rho'}$. At time $\tau$, the $\bsp'(A;\tree)$ process deletes a connected component $U_\rho$ containing $\rho$. The definition of $\bsp'$ implies that variables in $U_\rho$ can lie at depth at most $13R/10$ below $\rho$, and hence at distance at least $27R/10$ from $\rho'$. As a result, no other variables will be removed from $U$ until time $\tau_{\rho'}$, when the $\bsp'(A;\tree)$ process deletes a connected component $U_{\rho'}$ containing $\rho'$. Variables in $U_{\rho'}$ can lie at distance at most $13R/10$ from $ \rho'$, hence at distance at least $7R/10$ from $U_\rho$. No other variables will be removed from $U$ after time $\tau_{\rho'}$. In particular, this contradicts the assumption that $U$ is a tall block and $\bsp'(A;\tree)=\tree$. Very similar arguments give the desired contradiction in the cases $\tau=\tau_{\rho'}<\tau_\rho$ and $\tau=\tau_{\rho'}=\tau_\rho$. This proves our claim that any tall block must also be marked.\smallskip

\noindent\bemph{Step 4. Extraction of sparse subtree.} Now recall that the tree $\tree$ is rooted at a variable $\vrt$, from which there emanates a path $P$
of length $L\ge 8R$. 
Let $\ell\equiv \ell/(8R)$ and note that $\lfloor L/(4R)\rfloor \ge L/(4R)-1 \ge \ell$. It follows that there is a path of blocks $\mathfrak{P}=(U_1,\ldots,U_\ell)$, where $U_1$ is rooted at $\vrt$ and $U_i$ is a child block of $U_{i-1}$ for each $2\le i\le \ell$. By the claim proved in the previous step, each $U_i$ must be a marked block: that is, each $U_i$ either intersects $A$ or has at least two marked child blocks. It follows from the definition of the marking process that the path $\mathfrak{P}$ can be covered by a disjoint union
	\[\mathfrak{T}=\bigsqcup_{j=1}^s\mathfrak{T}_j\]
where each $\mathfrak{T}_j$ is a nonempty full binary tree of marked blocks, rooted at a block $U_{i(j)}\in\mathfrak{P}$, such that the leaves $\mathfrak{L}_j$ of $\mathfrak{T}_j$ are all blocks intersecting $A$.\footnote{We will explain the $\mathfrak{T}_j$ by an example. The root of $\mathfrak{T}_1$ is always $U_1$, so $i(1)=1$. Suppose that $U_1$ does not intersect $A$, in which case it must have two marked child blocks, say $U_2$ and $\tilde{U}_2$. Suppose that these do intersect $A$. Then $\mathfrak{T}_1$ consists of $U_1$, $U_2$, and $\tilde{U}_2$. We then take $\mathfrak{T}_2$ to be rooted at $U_3$, so in this case $i(2)=3$. Suppose $U_3$ also intersects $A$; then $\mathfrak{T}_2$ consists only of $U_3$. We take $\mathfrak{T}_3$ to be rooted at $U_4$, and so on.}
Let $\mathfrak{T}$ be the tree of blocks given by the union of $\mathfrak{T}_1,\ldots,\mathfrak{T}_s$; it has maximum degree at most four. We then construct a subtree $T\subseteq\tree$ as follows:
\begin{enumerate}[a.]
\item For $1\le j\le s$ let $B_j\subseteq A$ be defined by taking one variable from $U\cap A$ for each $U\in\mathfrak{L}_j$.

\item For $1\le j\le s$ let $T_j\subseteq\tree$ be defined by taking the union of all paths between $B_j$ and the root of $U_{i(j)}$.

\item For $1\le j\le s-1$, let $Q_j$ be the union of $T_j$ with a path joining the roots of $U_{i(j+1)-1}$ and $U_{i(j+1)}$. Let $Q_s$ be the union of $T_s$ with a path joining the root of $U_\ell$ with a variable of maximal depth in $U_\ell$. (The precise definition of $Q_s$ is not so important; we choose this one as it will guarantee $\diam T = L$.)

\item Let $T\subseteq\tree$ be the union of $Q_1,\ldots,Q_s$. Let $B'$ be the (disjoint) union of $B_1,\ldots,B_s$.
\end{enumerate}
Thus $T$ is a subtree of $\tree$ that includes the root $\vrt$ and has maximum degree at most four. We now argue that $T$ is $1/(80R)$-corrupted. The intersection of $T_j$ with any block $U$ is given by a union of at most two paths, so contains at most $8R$ variables. This implies
	\[
	|V_{T_j}|
	\le \sum_{U\in\mathfrak{T}_j}
		|V_{T_j}\cap U|
	\le 8R |\mathfrak{T}_j|
	\]
where $|\mathfrak{T}_j|$ denotes the number of blocks in $\mathfrak{T}_j$. We then have 
$|V_{Q_j}|\le |V_{T_j}|+4R
\le 8R (|\mathfrak{T}_j|+1)$, so
	\beq\label{e:corrupt.paths.ubd}
	|V_T|
	\le \sum_{j=1}^s |V_{Q_j}|
	\le 8R \sum_{j=1}^s \Big(|\mathfrak{T}_j|+1\Big)\,.
	\eeq
On the other hand, in any nonempty full binary tree, the number of leaf nodes is exactly one plus the number of internal nodes, so
	\beq\label{e:corrupt.leaves.lbd}
	|V_T\cap A|
	\ge |B'|
	= \sum_{j=1}^s|B_j|
	\ge\sum_{j=1}^s|\mathfrak{L}_j|
	= \sum_{j=1}^s \f{|\mathfrak{T}_j|+1}2\,.
	\eeq
Since $\mathfrak{T}$ has maximum degree at most four,
we can extract $B\subseteq B'$ with $|B| \ge |B'|/5 $ such that variables in $B$ lie at pairwise distance greater than $2R$. Combining with 
\eqref{e:corrupt.paths.ubd}~and~\eqref{e:corrupt.leaves.lbd} gives
	\[
	\f{|B|}{|V_T|}
	\ge
	\f{|B'|}{5|V_T|}
	\ge \f1{80R}\,,
	\]
so $T$ is $1/(80R)$-corrupted as claimed. (Note in this case that the condition that $T\cup B_R(B;\tree)$ be acyclic is trivially satisfied, since $\tree$ is a tree.)
\end{proof}
\end{lem}

In fact, by an essentially identical proof, we have the following slight generalization of Lemma~\ref{l:marked.blocks.corrupt.tree}, which will be used in the analysis that follows.

\begin{cor}\label{c:marked.blocks.corrupt.tree} Suppose $\GG=(V,F,E)$ is a finite bipartite factor graph in which all clauses have degree $k$, and which can be expressed as \[\GG=\tree\cup\GG'\] where $\GG'$ is an arbitrary graph, and $\tree$ is a tree that intersects $\GG'$ at a single variable $v$, such that $\tree$ has depth $L$ when rooted at $v$. If $A$ is a subset of $V$ such that $\bsp'(A;\GG)=\GG$, then there is a $1/(80R)$-corrupted subtree $T\subseteq\tree$ of maximum degree at most four, with $L \le \diam T \le \diam \tree \le 2L$.

\begin{proof}
The proof of Lemma~\ref{l:marked.blocks.corrupt.tree} applies to $\tree$; the only difference is that we fix $v$ to be the root of $\tree$. In the final step of extracting $B$ from $B'$, it is easy to arrange that none of the variables in $B$ lie in the topmost block of $\tree$. By construction, the tree $T$ has depth exactly $L$, hence diameter between $L$ and $2L$. Then the requirement that $T\cup B_R(B;\GG)$ be acyclic is satisfied, since this will be a subgraph of the tree $\tree$.
\end{proof}
\end{cor}

For the proof of Proposition~\ref{p:corrupted.tree.in.graph}, the main challenge remaining is to reduce to the case of Corollary~\ref{c:marked.blocks.corrupt.tree}
(or to its special case Lemma~\ref{l:marked.blocks.corrupt.tree}). To this end,
it is useful to consider a slight variant of $\bsp'$
(Definition~\ref{d:slowed.bspprime})
and prove a self-consistency property thereof
(Lemma~\ref{l:bspprime.within} below).

\begin{dfn}[slowed removal process $\bsp''$]
\label{d:slowed.bspprime}
Let $\GG=(V,F,E)$ be any bipartite factor graph (finite or infinite), and let $A$ be a finite subset of $V$.
Recall from Definition~\ref{d:proc} that the removal process $\bsp'(A;\GG)$ goes from $\GG_t$ to $\GG_{t+1}$ (for $t\ge-1$) by removing the $R$-neighborhoods of all the variables in $A_t$, where $A_{-1}$ is the initial set $A$, and $A_t$ is defined inductively using \eqref{e:bspprime.activated.set} as $\ACT(\GG_t)$. We now define the process $\bsp''(A;\GG)$ which is equivalent to $\bsp'(A;\GG)$ except that it removes one $R$-neighborhood at a time. That is to say, for each $t\ge-1$, we arbitrarily order the variables in $A_t$ as
	\[
	\bigg(v_{t,i} : 1\le i \le |A_t|\bigg)\,,
	\]
and then remove first 
$B_R(v_{t,1};\GG_t)$,
then $B_R(v_{t,2};\GG_t)$,
and so on. We say that $v_{t,i}$ is \textbf{visited} at the $i$-th step of this process (even if it may have been deleted at an earlier step). Thus $\bsp''$ reaches the graph $\GG_{t+1}$ in
	\[\sum_{s=-1}^t|A_s|\]
steps. The point of $\bsp''$ is that each step cannot increase the maximum component diameter (of the removed subgraph) by too much: if the maximum component diameter before a $\bsp''$ step is $\ell$, the maximum component diameter after the $\bsp''$ step is at most $2(\ell+R+1)$. We let $\bsp^s$ denote the $\bsp''$ process stopped after $s$ steps.
\end{dfn}

\begin{lem}\label{l:bspprime.within}
Let $\GG=(V,F,E)$ be any bipartite graph (finite or infinite) in which all clauses have degree $k$. Let $A$ be a finite subset of $V$. If for some finite $s$ we have $\bsp^s(A;\GG)=\HH$,
then $\bsp'(A_{\HH};\HH)=\HH$ where $A_{\HH}$ denotes the restriction of $A$ to $\HH$.

\begin{proof} First we note that $A_{\HH}$ is a strict subset of $A$ if and only if $\bsp^s$ stops before the initial round of $\bsp'(A;\GG)$ is finished, i.e., if and only if $s<|A_{-1}|$. In this case the result is straightforward: let $U_{-1}\subseteq A_{-1}$ denote the first $s$ variables in $A_{-1}$, so $\HH=\bsp^s(A;\GG)= B_R(U_{-1};\GG)$. This implies $U_{-1}\subseteq A_{\HH}$
and $B_R(U_{-1};\GG)=B_R(U_{-1};\HH)$. From this we obtain the chain of relations
	\[
	\HH\supseteq
	\bsp'(A_{\HH};\HH)
	\supseteq B_R(A_{\HH};\HH)
	\supseteq B_R(U_{-1};\HH)
	= B_R(U_{-1};\GG) = \HH\,,
	\]
and so we have $\bsp'(A_{\HH};\HH)=\HH$ as desired. We therefore assume from now on that $s\ge|A_{-1}|$. This means that $\bsp^s(A;\GG)$ completes the initial round of $\bsp'(A;\GG)$, producing $\GG_0 = \GG\setminus B_R(A;\GG)$.
It also means that $A_{\HH}=A$, and $\bsp'(A;\HH)$ completes its initial round to produce $\HH_0 = \HH\setminus B_R(A;\HH) = \HH\setminus B_R(A;\GG)$. We hereafter index this initial round as the ``zeroth round.''

We now suppose inductively, for $i\ge0$, the following hypothesis: provided that $\bsp^s(A;\GG)$ fully completes
the first $i$ rounds
 of $\bsp'(A;\GG)$ (again, $i$ starts from zero), the following hold:
\begin{enumerate}[(i)]
\item $A_t=\ACT(\GG_t)=\ACT(\HH_t)$ for all $t\le i-1$; and
\item The subgraph $\RR_i$ removed from $\GG$ within the first $i$ rounds of $\bsp'(A;\GG)$
agrees with the subgraph removed from $\HH$
 within the first $i$ rounds of $\bsp'(A;\HH)$, i.e.,
 	\beq\label{e:same.R.i}
	\RR_i
	\equiv\bigcup_{t=-1}^{i-1}
		B_R(A_t;\GG_t)
	=\bigcup_{t=-1}^{i-1}B_R(A_t;\HH_t)\,.
	\eeq
\end{enumerate} 
The base case $i=0$ follows from the preceding argument.
Let $\GG_i$ be the graph remaining after the first $i$ rounds of $\bsp'(A;\GG)$, and let $\HH_i$ be the graph remaining after the first $i$ rounds of $\bsp'(A;\HH)$. By induction, $\GG_i=\GG\setminus\RR_i$
and $\HH_i=\HH\setminus\RR_i = \GG_i \cap \HH$.
We will say that a subgraph is ``lacking'' if it contains at least two clauses of degree $k-1$, or at least one clause of degree $\le k-2$. Then, recalling \eqref{e:bspprime.activated.set}, 
for the next round we must consider
	\begin{align*}
	\ACT(\GG_i)
	&=\bigg\{\textup{variables $v$ in $\GG_i$
	such that $B_{3R/10}(v;\GG_i)$ is lacking}\bigg\}\,,\\
	\ACT(\HH_i)
	&=\bigg\{\textup{variables $v$ in $\HH_i$
	such that $B_{3R/10}(v;\HH_i)$ is lacking}
	\bigg\}\,.
	\end{align*}
The inductive hypothesis implies $\HH_i\subseteq\GG_i$,
so $B_R(v;\HH_i)\subseteq B_R(v;\GG_i)$ for any variable $v$. On the other hand, let $U_i$ be the vertices in $\ACT(\GG_i)$
that are visited by $\bsp^s(A;\GG)$
(for ``visited'' in the sense of Definition~\ref{d:slowed.bspprime}). If $v\in U_i$ then
the $\bsp^s(A;\GG)$ process removes
 $B_R(v;\GG_i)$, which means it lies inside $\HH$. This implies
$B_R(v;\GG_i)\subseteq
\GG_i \cap \HH=\HH_i$, where the last equality uses the inductive hypothesis. This proves that
	$B_R(v;\GG_i)=B_R(v;\HH_i)$
for all $v\in U_i$,
hence also $B_{3R/10}(v;\GG_i)=B_{3R/10}(v;\HH_i)$
for all $v\in U_i$. Moreover, all variables in $U_i$ are removed, so $U_i$ must be contained in $\GG_i\cap\HH=\HH_i$
(using the inductive hypothesis again). It follows from this that
	\[
	U_i
	= \bigg\{v\in U_i : 
	\textup{$B_{3R/10}(v;\GG_i)$ is lacking}\bigg\}
	\subseteq \bigg\{v\in \HH_i : 
	\textup{$B_{3R/10}(v;\HH_i)$ is lacking}\bigg\}
	=\ACT(\HH_i)\,.\]
Conversely, suppose $v\in\ACT(\HH_i)$, so $B_{3R/10}(v;\HH_i)$ is lacking: either there are two clauses $a_1,a_2 \in B_{3R/10}(v;\HH_i)$ of degree $k-1$, or there is a single clause $a\in B_{3R/10}(v;\HH_i)$ of degree $\le k-2$. We have by induction $\HH_i\subseteq\GG_i$, therefore $B_{3R/10}(v;\HH_i)\subseteq B_{3R/10}(v;\GG_i)$ which means these clauses are also present in $B_{3R/10}(v;\GG_i)$. All clauses have degree $k$ in $\GG$ (and hence also in $\HH$), so if a clause $a\in B_{3R/10}(v;\HH_i)$
has degree $k-j$ relative to $\HH_i$, there must be $j$ edges joining that clause to variables in $\RR_i$. It then follows from the inductive hypothesis \eqref{e:same.R.i}
that the clause has degree $\le k-j$ in $\GG_i$. This implies that $B_{3R/10}(v;\GG_i)$ must also be lacking,
so $\ACT(\HH_i)\subseteq\ACT(\GG_i)$. If the $(i+1)$-st round is completed, then we obtain $A_i=\ACT(\GG_i) = U_i = \ACT(\HH_i)$. Since we proved above that $B_R(v;\GG_i)=B_R(v;\HH_i)$ for all $v\in U_i = A_i$, we conclude
$\RR_{i+1}= \RR_i \cup B_R(A_i;\GG_i)=B_R(A_i;\HH_i)$.
This verifies the inductive hypothesis.

If $\bsp^s(A;\GG)$ completes exactly the first $i$ rounds of $\bsp'(A;\GG)$ (and no more), then the above induction implies $\HH=\RR_i=\bsp'(A;\HH)$, as desired. It remains finally to consider the case that $\bsp^s(A;\GG)$ completes the first $i$ rounds, and does not complete the $(i+1)$-st round. In this case, $\HH=\bsp^s(A;\GG) = \RR_i \cup B_R(U_i;\GG_i)$. The above argument also gives
$U_i\subseteq \ACT(\HH_i)$. Altogether we obtain the chain of relations
	\[
	\HH\supseteq \bsp'(A;\HH)
	\supseteq\RR_i \cup B_R(\ACT(\HH_i);\HH_i)
	\supseteq\RR_i \cup B_R( U_i ;\HH_i)
	=\HH\,.
	\]
This gives $\bsp'(A;\HH)=\HH$ in the case where $\bsp^s(A;\GG)$ stops partway through some round of $\bsp'(A;\GG)$,
and concludes the proof of the lemma.
\end{proof}
\end{lem}

We are now finally prepared to prove the main result of this subsection:

\begin{proof}[Proof of Proposition~\ref{p:corrupted.tree.in.graph}] Instead of $\bsp'(A;\GG)$, we follow the slowed removal process $\bsp''(A;\GG)$ (Definition~\ref{d:slowed.bspprime}) where we remove one $R$-neighborhood at a time. Recall that if the maximum component diameter before a $\bsp''$ step is $\ell$, then the maximum component diameter after the $\bsp''$ step is at most $2(\ell+R+1)$. Let $\TIME_\star$ be the first time that $\bsp''(A;\GG)$ creates any component $\HH\subseteq\mathscr{B}$ of diameter at least $5L$. The maximum component diameter at time $\TIME_\star-1$ is at most $5L-1$, so $\HH$ can have diameter at most $2[(5L-1)+R+1]\le 11L$. In particular, if $\HH$ contains more than one cycle, we are in scenario (i) and the conclusion follows.

If $\HH$ is a tree, then Lemma~\ref{l:bspprime.within} gives $\bsp'(A_{\HH};\HH)=\HH$ where $A_{\HH}$ is the restriction of $A$ to $\HH$. It follows from Lemma~\ref{l:marked.blocks.corrupt.tree} (with $\tree=\HH$) that there is a subtree $T\subseteq\HH$ with
 maximum degree at most four and
	\[\diam T = \diam \HH \in[5L,11L]\,,\]
which is $1/(80R)$-corrupted with respect to $(A_{\HH},\HH)$ --- in particular, this gives that $T\cup B_R(A_{\HH};\HH)$ is acyclic. Note however that the definition of the removal process implies $B_R(A_{\HH};\GG)\subseteq\HH$, and so $T\cup B_R(A_{\HH};\HH)$ is in fact the same as $T\cup B_R(A_{\HH};\GG)$. It follows that $T$ is also $1/(80R)$-corrupted with respect to $(A_{\HH},\GG)$, and hence also with respect to $(A,\GG)$. This proves that scenario (ii) holds if $\HH$ is a tree.

Next we consider the case that $\HH$ contains a single cycle $C$ with $\diam C \ge 5L/2$. Let $\TIME'$ be the first time that the $\bsp''(A;\GG)$ process created any component $\HH'\subseteq\HH$ of diameter at least $L$. Then $\TIME'<\TIME$, and $\HH'$ has diameter at most $2[(L-1)+R+1] < 5L/2$. It follows that $\HH'$ must be a tree. Lemma~\ref{l:bspprime.within} gives
$\bsp'(A_{\HH'};\HH')=\HH'$. Arguing as for the previous case, we apply Lemma~\ref{l:marked.blocks.corrupt.tree} (now $\tree=\HH'$) to see that there is a subtree $T\subseteq\HH'$ with
 maximum degree at most four and
 	\[\diam T = \diam \HH' \in \bigg[L,\f{5L}{2}\bigg]\,,\]
which is $1/(80R)$-corrupted with respect to $(A_{\HH'},\HH')$, and hence also with respect to $(A,\GG)$. This proves that scenario (ii) also holds in this case.

Finally, in the case that $\HH$ contains a single cycle $C$ with $\diam C < 5L/2$, we must be able to decompose
	\[\HH = \tree' \cup \HH''\]
where $C\subseteq \HH''$, and $\tree'$ is a tree that intersects $\HH''$ at a single variable $v$, such that $\tree'$ has depth exactly $L$. The conditions of Corollary~\ref{c:marked.blocks.corrupt.tree} are then satisfied, so we conclude that there is a $1/(80R)$-corrupted subtree $T\subseteq\tree'$ of maximum degree at most four, with $L \le \diam T \le \diam \tree' \le 2L$. Thus scenario (ii) again holds.
\end{proof}

\subsection{Probabilistic analysis of preprocessing} 
\label{ss:preprocessing.probabilistic}

First we use Lemma~\ref{l:radon.derivative.pgwT}
to transfer the 
result of Corollary~\ref{c:excellent}
(from \S\ref{ss:orderly.contained}) from the $\uPGW$ measure to the $\uPGW(T)$ measure of Definition~\ref{d:PGW.based.on.T}, where $T$ is any fixed sparse tree: 

\begin{cor}\label{c:excellent.under.uPGW.T}
Let $T$ be any fixed variable-rooted tree of maximum degree $\CC$,
where $\CC$ is an absolute constant.
Then
	\[
	\Big[\uPGW(T)\Big]\bigg(
	\textup{$\vrt$ is improper or not $\II$-excellent}
	\bigg)
	\le \f1{\exp( 2^{k\DELTACONST/4} R / k )}\]
for $\II\in\set{0,1}$ and $k\ge k_0$, where $k_0$ is an absolute constant depending only on $\DELTACONST$ and $\CC$.

\begin{proof}
As in the proof of Lemma~\ref{l:radon.derivative.pgwT},
let $\RN$ denote the Radon--Nikodym derivative 
of $\uPGW(T)$ with respect to $\uPGW$.
Let $\E$ denote expectation with respect to $\uPGW$.
For any event $E$ we can trivially bound
	\[
	\E\bigg[ \mathbf{1}_E \RN ; \RN \le \f1{\uPGW(E)^{2/3}}\bigg]
	\le \uPGW(E)^{1/3}\,.
	\]
On the other hand we can use Lemma~\ref{l:radon.derivative.pgwT}
together with Markov's inequality to bound
	\[
	\E\bigg[ \RN ; \RN > \f1{\uPGW(E)^{2/3}}\bigg]^2
	\le \E\Big[\RN^2\Big]
			\uPGW\bigg(\RN \ge \f1{\uPGW(E)^{2/3}}\bigg)
	\le 2 (\E\RN) \, \uPGW(E)^{2/3}
	= 2 \, \uPGW(E)^{2/3}\,.
	\]
Combining the bounds gives
	\[\Big[\uPGW(T)\Big](E)
	=\E\Big[\mathbf{1}_E\RN\Big]
	\le \E\bigg[ \mathbf{1}_E \RN ; \RN \le \f1{\uPGW(E)^{2/3}}\bigg]
		+\E\bigg[ \RN ; \RN > \f1{\uPGW(E)^{2/3}}\bigg]
	\le 3 \, \uPGW(E)^{1/3}\,.
	\]
The result then follows by combining with Corollary~\ref{c:excellent}.
\end{proof}
\end{cor}

We can use the above corollary to bound the probability, under the $\uPGW$ measure, that the root lies in a sparse tree that is corrupt (see Definition~\ref{d:corrupt}) with respect to the set of non-excellent vertices. Recall \eqref{e:bold.Lambda.abstract.defn} that $\bLambda_{\CC,s}$ refers to $\CC$-sparse rooted trees having $s$ variables. Also recall \eqref{e:bold.Lambda.within.PGW.defn} that $\bLambda_{\CC,s}(\tree)$ refers to $\CC$-sparse subtrees $\vrt\in T\subseteq\tree$ having $s$ variables.

\begin{lem}\label{l:PGW.bound.on.corrupt.tree}
Let $\CC$ be an absolute constant. For $\tree\sim\uPGW$ let $A(\tree)$ denote the set of all variables in $\tree$ that are improper or not $\II$-excellent. It holds for all $s\ge 100R$ that
	\[
	\uPGW\bigg(
	\begin{array}{c}
	\textup{some $T\in\bLambda_{\CC,s}$ is $1/(80R)$-corrupt}\\
	\textup{with respect to $(\tree,A(\tree))$}
	\end{array}
	\bigg)
	\le \f1{\exp(2^{k\DELTACONST/4} s/k^2)}\]
provided $k\ge k_0$, where $k_0$ is an absolute constant
depending only on $\DELTACONST$ and $\CC$.

\begin{proof}
This argument is similar to (but simpler than) part of the proof of Lemma~\ref{lem-local-D-0}.
First we fix a tree $\tprime\in\bLambda_{\CC,s}$, as well as a subset $B'\subseteq V_T$
with $|B'| \ge |V_T|/(80R)$ such that all variables in $B'$ lie at pairwise distance greater than $2R$.
For this fixed pair $(\tprime,B')$ we bound (cf.\ \eqref{e:expected.number.of.embeddings}, \eqref{e:embeddings.times.embedded.probab}, and \eqref{e:product.over.disjoint.nbhds.around.sparse.tree})
	\begin{align*}
	P(T',B')
	&\equiv\uPGW\bigg(
	\textup{there exists an embedding $\zeta:\tprime\hookrightarrow\tree$
	such that $\zeta(B')\subseteq A(\tree)$}\bigg) \\
	&\le\int\bigg|\bigg\{
	\textup{embeddings $\zeta:\tprime\hookrightarrow\tree$
		such that $\zeta(B')\subseteq A(\tree)$}
	\bigg\}\bigg|\,d\uPGW(\tree)\\
	&= \embed(T')
	\cdot \Big[\uPGW(T')\Big]\bigg( B'\subseteq A(\tree) \bigg)
	\le (\alpha k^2)^{\CC s}\cdot
	\prod_{u\in B'}\bigg\{
	\Big[\uPGW(T_u)\Big]\bigg( \textup{root of $\tree'$ 
		is in $A(\tree')$} \bigg)\bigg\}\,,
	\end{align*}
where $T_u$ now refers to the tree $\tprime$ rerooted at $u$, $\tree'$ refers to a sample from the measure $\uPGW(T_u)$, and the probability factorizes over $u\in B'$ since ``proper'' and ``$\II$-excellent''
are properties of the $R$-neighborhood, and we have assumed that variables in $B'$ lie at pairwise distance greater than $2R$. Applying Corollary~\ref{c:excellent.under.uPGW.T} (for the measures $\uPGW(T_u)$) and summing over all possibilities of $(T',B')$ gives
	\[\sum_{T',B'}P(T',B')
	\le \f{e^{O(\CC k s)}}{ \exp(2^{k\DELTACONST/4} s / (80k) ) }
	\le \f1{ \exp(2^{k\DELTACONST/4} s/k^2) }\,.\]
This implies the assertion of the lemma.
\end{proof}
\end{lem}

We next transfer the bound of Lemma~\ref{l:PGW.bound.on.corrupt.tree}
from the tree measure $\uPGW\equiv\uPGW^\alpha$ to the random $\ksat$ measure $\poisP\equiv\poisP^{n,\alpha}$ (as given in Definition~\ref{d:formal.random.ksat}, and including the random marking
$\LABEL_R$ of \eqref{e:random.R.marking}). For $\GG=(V,F,E)\sim\P$ 
and any $v\in V$, let $\bLambda_{\CC,s}(\GG,v)$ denote the set of subtrees $v\in T\subseteq\GG$ having $s$ variables and maximum degree bounded by $\CC$. 

\begin{lem} \label{l:corrupt.tree.in.erdos.renyi}
Let $\CC$ be an absolute constant. For $\GG=(V,F,E)\sim\poisP\equiv\poisP^{n,\alpha}$ let $A(\GG)$ denote the set of all variables in $\GG$ that are improper or not $\II$-excellent, for $\II\in\set{0,1}$. For any $v\in V$, and for all $100R \le s \le n^{1/10}$, 
	\[\P\bigg(\begin{array}{c}
	\textup{some $T\in\bLambda_{\CC,s}(\GG, v)$
	is $1/(80R)$-corrupt}\\
	\textup{with respect to $(\GG,A(\GG))$}
	\end{array}
	\bigg) \le
	\f{ e^{O(\CC k s)}}{ \exp(2^{k\DELTACONST/4} s / (80k) ) }\,.
	\]

\begin{proof} Fix a tree $\tprime\in\bLambda_{\CC,s}$, as well as a subset $B'\subseteq V_{\tprime}$ with $|B'| \ge |V_{\tprime}|/(80R)$ such that all variables in $B'$ lie at pairwise distance greater than $2R$. We consider the event
	\[
	E_n(\tprime,B')
	\equiv\bigg\{\hspace{-3pt}\begin{array}{c}
	\textup{there exists an embedding
	$\zeta:\tprime\hookrightarrow\GG$ such that}\\
	\textup{$\zeta(B')\subseteq A(\GG)$ and
		$\zeta(\tprime) \cup B_R( \zeta(B');\GG)$ is acyclic}
	\end{array}\hspace{-3pt}
	\bigg\}\,.
	\]
Write $P_n(\tprime,B') \equiv \poisP(E_n(\tprime,B'))$
and $P_{n,m}(\tprime,B') \equiv \P_{n,m}(E_n(\tprime,B'))$. Recalling \eqref{e:poisson.erdos.renyi}, we can bound
	\[P_n(\tprime,B')
	\le\sum_{m\ge1} \mathbf{1}\bigg\{
		\f{|m-n\alpha|}{n^{1/2}\log n} \le1 \bigg\}
		P_{n,m}(\tprime,B')
	+ \f1{\exp\{(\log n)^{3/2}\}}\]
--- the last term accounts for the probability under $\poisP^{n,\alpha}$ that the total number of clauses in $\GG$ (a $\Pois(n\alpha)$ random variable) deviates from $n\alpha$ by more than $n^{1/2}\log n$.
We now fix any $m$ satisfying
	\[\Big|m-n\alpha\Big|\le n^{1/2}\log n\,,\]
and consider $\GG$ sampled from $\P_{n,m}$.
Let $\zeta'$ denote any mapping that sends $V_{\tprime} \hookrightarrow [n]$, with $\zeta'(\vrt)=v$, and $F_{\tprime} \hookrightarrow [m]$; and let $\bm{j}$ denote a mapping $E_{\tprime} \to [k]$. We then write $\set{\zeta' : \tprime \hookrightarrow_{\bm{j}} \GG}$ for the event that the pair $(\zeta',\bm{j})$ is consistent with an actual embedding of $\tprime$ into $\GG$: this means that for every edge $e=(av)\in E_{\tprime}$, the edge $\zeta'(e)\equiv (\zeta'(a)\zeta'(v))$ is present in $E_{\GG}$, and the index of this edge in the clause $\zeta'(a)$ is given by $\bm{j}(e) \in[k]$. With this notation we can bound
	\beq\label{e:condition.on.number.of.clauses}
	P_{n,m}(\tprime,B')
	\le\sum_{\zeta',\bm{j}}
	\P_{n,m}\bigg(\zeta' : \tprime \hookrightarrow_{\bm{j}} \GG\bigg)
	\P_{n,m}\bigg(\hspace{-3pt}\begin{array}{c}
	\textup{$\zeta'(B') \subseteq A(\GG)$ and}\\
	\textup{$\zeta'(\tprime) \cup B_R( \zeta'(B');\GG)$ is acyclic}
	\end{array}\hspace{-3pt}
	\,\bigg|\,\zeta' : \tprime \hookrightarrow_{\bm{j}} \GG\bigg)\,.
	\eeq
In words, we interpret the above decomposition as follows. To sample $\GG\sim\P_{n,m}$, we start with $n$ variables and $m$ clauses where each clause is equipped with $k$ outgoing edges (indexed $j=1,\ldots,k$). Each outgoing edge matches to a uniformly random variable, independently of all other edges --- this means that we can sample $\GG$ in a \emph{sequential} way, revealing one edge at a time. In particular, when we fix a pair $(\zeta',\bm{j})$ and condition on the event $\set{\zeta' : \tprime \hookrightarrow_{\bm{j}} \GG}$, it is equivalent to say that for all $a\in F_{\tprime}$ and all $e=(av)\in\delta a$, we reveal that the $\bm{j}(e)$-th edge incident to clause $\zeta'(a)$ matches to the variable $\zeta'(v)$. Thus
	\[\P_{n,m}\bigg(\zeta' : \tprime \hookrightarrow_{\bm{j}} \GG\bigg)
	= \f1{n^{|E_{\tprime}|}}\,,\]
and conditioning on the event $\set{\zeta' : \tprime \hookrightarrow_{\bm{j}} \GG}$ reveals nothing about the remaining $mk-|E_{\tprime}|$ edges in the graph $\GG$. 
Note also that the total number of pairs $(\zeta',\bm{j}')$ is upper bounded by $\smash{n^{|V_{\tprime}|-1}m^{|F_{\tprime}|}k^{|E_{\tprime}|}}$.

Now fix $(\zeta',\bm{j})$ and condition on the event $\set{\zeta' : \tprime \hookrightarrow_{\bm{j}} \GG}$. We next want to explore the $R$-neighborhoods of the variables in $\zeta'(B')$. Note that this can also be done sequentially, in a ``breadth-first search'' manner:
first take any $u\in \zeta'(B')$, and reveal all the edges incident to it.
The degree of $u$ is then
	\[
	|\delta u|
	\sim
	|\delta_{\zeta'(\tprime)} u|
	+ \textup{Bin}\bigg( mk-|E_{\tprime}|, \f1n\bigg)\,,\]
where $|\delta_{\zeta'(\tprime)} u|$ is the degree of $u$ in the subgraph $\zeta'(\tprime)$. We then proceed to reveal the edges incident to the neighboring clauses of $u$, and so on until we have explored the entire subgraph
	\[\HH \equiv\zeta'(\tprime) \cup B_R(\zeta'(B'))\,. \]
If at any point in the exploration we reveal a cycle inside $\HH$, we can simply stop because this means the event of interest does not occur.
It is extremely unlikely to reveal any variable degree larger than $(\log n)^2$, so if this occurs we also simply halt the exploration.
Restricted to the event that no cycle is formed and all revealed variable degrees are at most $(\log n)^2$,
the law of $\HH$ can be bounded in terms of the $\uPGW^\alpha$ law: more precisely, we claim that
	\begin{align}\nonumber
	&\P_{n,m}\left(
	\Big\{\zeta'(B')\subseteq A(\GG)\Big\}
	 \cap \left\{\hspace{-3pt}
	\begin{array}{c}
	\textup{exploration of $\HH$ completes}\\
	\textup{without revealing any cycle or}\\
	\textup{variable of degree larger than $(\log n)^2$}
	\end{array}\hspace{-3pt}
	\right\}
	\right)\\
	&\qquad\le 2 
	\prod_{u\in B'}\bigg\{ \Big[\uPGW(T_u) \Big]\bigg(
	\textup{root of $\tree'$
	is in $A(\tree')$} \bigg)\bigg\}\,,
	\label{e:binom.stoch.dom.by.poisson}
	\end{align}
where $T_u$ is the tree $\tprime$ rerooted at $u$.
The bound \eqref{e:binom.stoch.dom.by.poisson} can be justified by noting that for all integers $m$ and $j$ with $|m-n\alpha| \le n^{1/2}\log n$ and all $0\le j \le (\log n)^2$, we have
	\[
	\f{\P( \textup{Bin}(mk,1/n) =j)}
	{\P( \Pois(\alpha k) =j)}
	= \binom{mk}{j} \f1{n^j} \bigg(1-\f1n\bigg)^{mk-j}
	\f{e^{\alpha k} j!}{(\alpha k)^j}
	\le \bigg(\f{n}{n-1}\bigg)^j
	=\exp\bigg\{ O\bigg( \f{(\log n)^2}{n}\bigg)\bigg\}\,.
	\]
If the exploration of $\HH$ completes without revealing any variable of degree larger than $(\log n)^2$, then the total number of variables explored is upper bounded by $n^{1/9}$ (for $n$ large enough), so the joint distribution of degrees will be close to i.i.d.\ $\Pois(\alpha k)$ random variables, and this gives \eqref{e:binom.stoch.dom.by.poisson}.
Applying Corollary~\ref{c:excellent.under.uPGW.T}
and summing over $(\zeta',\bm{j})$ gives
	\[
	P_{n,m}(\tprime,B')
	\le\f{n^{|V_{\tprime}|-1}m^{|F_{\tprime}|}k^{|E_{\tprime}|}}
		{ n^{|E_{\tprime}|} \exp(2^{k\DELTACONST/4} s / (80k) ) }
	=\f{\alpha^{|F_{\tprime}|} k^{|E_{\tprime}|}}
		{\exp(2^{k\DELTACONST/4} s / (80k) ) }
	\le \f{e^{O(\CC k s)}}{\exp(2^{k\DELTACONST/4} s / (80k) ) }\,.
	\]
Finally, summing over all $(T',B')$ proves the claim.
\end{proof}
\end{lem}

\begin{lem}\label{l:nbhd.growth.in.erdos.renyi} If $1\le L \le (\log n)/k^2$ and $\gamma \le \exp(-k^2 L)/k$, then
	\[
	\P\bigg( 
	\hspace{-3pt}\begin{array}{c}
	\textup{there exists a subset $S\subseteq V$ with}\\
	\textup{$|S|=n\gamma$ and $|B_L(S)|\ge n\gamma \exp(k^2 L)$}
	\end{array}\hspace{-3pt}
	\bigg)
	\le \f1{\exp(n e^{kL} \gamma)}\,.
	\]

\begin{proof}
As in the proof of Lemma~\ref{l:corrupt.tree.in.erdos.renyi}, we fix $m$ satisfying \eqref{e:condition.on.number.of.clauses} and consider $\GG\sim\P\equiv\P_{n,m}$. We then fix a subset $S\subseteq V$ of size $|S|=n\gamma$, and explore its neighborhood by breadth-first search. Let $Z_\ell(S)$ denote the number of variables at distance exactly $\ell$ from $S$. Note that
	\beq\label{e:bfs.in.graph.bound}
	\E\bigg( \exp \bigg\{\f{ Z_\ell(S) }{(\alpha k^2)^\ell}\bigg\}
		\,\bigg|\, B_{\ell-1}(S) \bigg)
	\le
	\E\bigg( \exp \bigg\{\f{
	\textup{Bin}\Big(mk, \f{Z_{\ell-1}(S)}{n-|B_{\ell-1}(S)|} \Big) }{(\alpha k^2)^\ell}\bigg\}
		\,\bigg|\, B_{\ell-1}(S) \bigg)\,.\eeq
Let $F_\ell$ be the event that $|Z_\ell(S)| \le n\gamma \exp(k^{3/2}{L})$, and let $E_\ell \equiv F_1 \cap \cdots \cap F_\ell$. On the event $E_{\ell-1}$ the neighborhood $B_{\ell-1}(S)$ cannot be too large (crudely, by the assumption on $\gamma$, it contains less than $1/k$ fraction of $V$), and it follows from \eqref{e:bfs.in.graph.bound} that
	\[
	\E\bigg( \exp \bigg\{\f{ Z_\ell(S) }{(\alpha k^2)^\ell}\bigg\}
		; E_{\ell-1} \bigg)
	\le 
	\E\bigg(
	\exp \bigg\{ \f{ Z_{\ell-1}(S) O(1/k) }{(\alpha k^2)^{\ell-1}}
		\bigg\}
		; E_{\ell-1} \bigg)
	\le
	\E\bigg(
	\exp \bigg\{\f{ Z_{\ell-1}(S) }{(\alpha k^2)^{\ell-1}}\bigg\}
		; E_{\ell-1} \bigg) \le e\,,
	\]
where the last bound holds by induction on $\ell$. It then follows by Markov's inequality that
	\[
	\P\Big( E_{\ell-1} \Big\backslash F_\ell \Big)
	\le \P\bigg( 
	\exp \bigg\{\f{ Z_\ell(S) }{(\alpha k^2)^\ell}\bigg\}
	\ge \exp \bigg\{ \f{n\gamma \exp(k^{3/2}L)}{(\alpha k^2)^\ell}\bigg\}
	 ; E_{\ell-1}\bigg)
	\le \f{e}{\exp(n\gamma \exp(k^{3/2}L/2))}\,.
	\]
Summing the last bound over $1\le\ell\le L$ gives
	\[
	\P\Big((E_L)^c\Big)
	= \sum_{\ell=1}^L \P\Big( E_{\ell-1} \Big\backslash F_\ell\Big)
	\le 
		 \f{L e}{\exp(n\gamma \exp(k^{3/2}L/2))}
	\le \f1{\exp(n\gamma \exp(k^{3/2} L/3) )}\,.
	\]
The claim follows by enumerating over at most $2^{n\gamma}$ choices for the subset $S\subseteq V$.
\end{proof}
\end{lem}

\begin{proof}[Proof of Proposition~\ref{p:small.fraction.removed.in.processing}] Fix $m$ satisfying \eqref{e:condition.on.number.of.clauses} and consider $\GG\sim\P\equiv\P_{n,m}$. For $v\in V$ we let $Y_v$ be the indicator that $v$ lies within distance $10R$ of a variable that is removed during processing, i.e.,
	\[
	Y_v=\mathbf{1}\bigg\{
		v\in B_{10R}\Big(\bsp'(A;\GG);\GG\Big)
		\bigg\}
	\]
where $A$ is the set of all variables that are improper or not $1$-good in $\GG$ (Definition~\ref{d:proc}). Let
	\[
	\bar{Y}_v
	=\mathbf{1}\bigg\{
	\begin{array}{c}
	\textup{$v$ lies within distance $10R$ of 
	a connected component of}\\
	\textup{$\bsp'(A;\GG)$ having diameter at most 
	$L_{\bsp'}\equiv (\log n)/2^{k\DELTACONST/5}$}
	\end{array}
	\bigg\}
	\le Y_v\,.
	\]
We first argue that, with high probability, $Y_v=\bar{Y}_v$ for all $v\in V$.
Indeed, Proposition~\ref{p:corrupted.tree.in.graph} implies that if 
$\bsp'(A;\GG)$ has a connected component of diameter more than $L_{\bsp'}$, then either (i) there is a connected subgraph $B'\subseteq\GG$ with $\diam B' \le 11 L_{\bsp'}/5$ that contains more than one cycle, or (ii) there is a $1/(80R)$-corrupted subtree $T'\subseteq\GG$ that has maximum degree at most four and $\diam T \ge L_{\bsp'}/5$. Thus
	\begin{align*}
	\P\bigg(
	\hspace{-3pt}\begin{array}{c}
	Y_v\ne \bar{Y}_v\\ \textup{ for any }v\in V
	\end{array}\hspace{-3pt}\bigg)
	&\le
	\sum_{v\in V} \bigg\{
	\P\bigg(\hspace{-3pt}\begin{array}{c}
	\textup{$B_{11L_{\bsp'}/5}(v)$ contains}\\
	\textup{more than one cycle}
	\end{array}\hspace{-3pt}\bigg)
	+\P\bigg(\hspace{-3pt}\begin{array}{c}
	\textup{some $T\in\bLambda_{4, L_{\bsp'}/5 }(\GG, v)$}\\
	\textup{is $1/(80R)$-corrupt}
	\end{array}\hspace{-3pt}\bigg)
	\bigg\}\\
	&\stackrel{\odot}{\le}
	n \bigg\{
	\bigg( \f{\exp(5 L_{\bsp'} k^2)}{n}\bigg)^2
	+
	\f1{\exp\{ 2^{k\DELTACONST/4} L_{\bsp'}/ k^2 \}}
	\bigg\}
	=o_n(1)\,,
	\end{align*}
where the bound marked $\odot$ follows from Lemma~\ref{l:corrupt.tree.in.erdos.renyi} together with very crude bounds on the chance to see more than one cycle in a neighborhood of $v$ of radius $5L_{\bsp'}$. This proves that $Y=\bar{Y}$ with high probability where
	\[
	Y\equiv\sum_{v\in V} Y_v\,,\quad
	\bar{Y}\equiv \sum_{v\in V} \bar{Y}_v\,.
	\]
We now turn to estimating $\bar{Y}$. Let $B'(\ell)$ be the union of all connected components of $\bsp'(A;\GG)$ of diameter $\ell$ that contain more than one cycle; and let $T'(\ell)$ be the union of all subtrees of $\bsp'(A;\GG)$ of diameter at least $\ell$ that are $1/(80R)$-corrupted. Let $A'(\ell)\equiv B'(\ell) \cup T'(\ell/5)$. Proposition~\ref{p:corrupted.tree.in.graph} implies that if $\bar{Y}_v=1$, then either $v$ lies within distance $2010R$ from $A$, or $v$ lies within distance $\ell+10R$ from $A'(\ell)$ for some $2000R\le\ell \le L_{\bsp'}$. Therefore
	\beq\label{e:expected.size.of.removed.component}
	\E\bar{Y}
	\le
	\E \Big|B_{2010R}(A)\Big|
	+\sum_{\ell= 2000R}^{L_{\bsp'}} \E \Big|B_{\ell+10R}(A'(\ell))\Big|\,.
	\eeq
Let $\gamma=\exp(-2^{k\DELTACONST/4} R / (2k))$, and note that 
Corollary~\ref{c:excellent.under.uPGW.T} together with Markov's inequality gives
	\[
	\P\Big(|A|\ge n\gamma\Big)
	\le
	\f{\E A}{n\gamma}
	= \f{\P(v\in A)}{\gamma}
	\le 
	\f{\exp(2^{k\DELTACONST/4} R / (2k))}
	{\exp(2^{k\DELTACONST/4} R / k)}
	= \gamma\,.
	\]
Combining with Lemma~\ref{l:nbhd.growth.in.erdos.renyi}
gives, for $L=2010R$ and $\gamma=\exp(-2^{k\DELTACONST/4} R / (2k))$,
	\begin{align}\nonumber
	\E \Big|B_{2010R}(A)\Big|
	&\le 
	n\gamma \exp(k^2 L)
	+n\P\Big( |A| \ge n \gamma\Big)
	+n
	\P\bigg( 
	\hspace{-3pt}\begin{array}{c}
	\textup{there exists $S\subseteq V$ with $|S|=n\gamma$}\\
	\textup{and $|B_L(S)|\ge n\gamma \exp(k^2 L)$}
	\end{array}\hspace{-3pt}
	\bigg) \\
	&\le
	\f{n \exp(2010k^2 R)}
		{\exp(2^{k\DELTACONST/4} R / (2k)) }
	+ \f{n}{ \exp(2^{k\DELTACONST/4} R / (2k)) }
	+ \f{n}{\exp(n e^{kL} \gamma)}
	\le \f{n}{ \exp(2^{k\DELTACONST/4} R / (3k)) }\,.
	\label{e:bound.nbd.of.A}
	\end{align}
Similarly, if we take $\gamma_\ell=\exp(-2^{k\DELTACONST/4}\ell/ k^2)$, then Lemma~\ref{l:corrupt.tree.in.erdos.renyi} together with Markov's inequality gives
	\[
	\P(|A'(\ell)| \ge n\gamma_\ell)
	\le \f{\E|A'(\ell)|}{n\gamma_\ell}
	\le \f{\P(v\in A'(\ell))}{\gamma_\ell}
	\le o_n(1) + \f{\exp( 2^{k\DELTACONST/4}\ell/ k^2 )}
				{\exp( 2^{k\DELTACONST/4} \Omega(\ell/ k ))}
	\le \gamma_\ell\,,
	\]
where the $o_n(1)$ error term accounts for the probability that $B_\ell(v)$ contains more than one cycle. Combining with Lemma~\ref{l:nbhd.growth.in.erdos.renyi} gives
	\begin{align}\nonumber
	\E\Big| B_{\ell+10R}(A'(\ell))\Big|
	&\le n\gamma_\ell \exp(k^2(\ell+10R))
	+n\P\Big(|A'(\ell)| \ge n\gamma_\ell\Big)
	+n\P\bigg( 
	\hspace{-3pt}\begin{array}{c}
	\textup{there exists $S\subseteq V$ with
		$|S|=n\gamma_\ell$ and}\\
	\textup{$|B_{\ell+10R}(S)|\ge n\gamma_\ell
		\exp(k^2(\ell+10R))$}
	\end{array}\hspace{-3pt}
	\bigg) \\
	&\le
	\f{n \exp(k^2(\ell+10R))}{\exp(2^{k\DELTACONST/4}\ell/ k^2)}
	+ \f{n}{\exp(2^{k\DELTACONST/4}\ell/ k^2)}
	+ \f{n}{\exp(n e^{k(\ell+10R)} \gamma_\ell)}
	\le \f{n}{\exp(2^{k\DELTACONST/4}\ell/ k^3)}\,,
	\label{e:bound.nbd.of.A.ell}
	\end{align}
for all $\ell\ge2000R$. Substituting \eqref{e:bound.nbd.of.A} and \eqref{e:bound.nbd.of.A.ell} into \eqref{e:expected.size.of.removed.component} gives
	\[
	\E \bar{Y}
	\le \f{n}{\exp(2^{k\DELTACONST/4} R / k^4)}\,.
	\]
Next note that under $\P_{n,m}$, for any $u\ne v$ the radius-$(L_{\bsp'}+10R)$ neighborhoods of $u$ and $v$ intersect with chance $o_n(1)$, and so are nearly independent of one another. If follows that $\Cov(\bar{Y}_u,\bar{Y}_v) = o_n(1)$. Therefore
	\[
	\Var\bar{Y} \le n
	+ \sum_{u\ne v} \Cov(\bar{Y}_u,\bar{Y}_v) = o_n(n^2)\,.
	\]
We can then use Chebychev's inequality to conclude that
	\[
	\P\bigg(\bar{Y} \ge \f{n}{\exp(2^{k\DELTACONST/5} R)} \bigg)
	\le
	\P\Big(\bar{Y} \ge 2 \E\bar{Y} \Big)
	\le \f{\Var \bar{Y}}{ (\E\bar{Y})^2} = o_n(1)\,.
	\]
Since we already argued above that $\P(Y\ne\bar{Y})=o_n(1)$, the result follows.\end{proof}

\subsection{Combinatorial analysis for positive type fractions}\label{ss:positive.fraction.combinatorial}

\noindent We now turn to the proof of Proposition~\ref{p:posfrac}. In this subsection we prove a deterministic result, Corollary~\ref{c:event.implies.desired.type}, which says essentially that if the local neighborhood of a clause $a$ in $\GG$ satisfies certain properties (to be detailed below), then we can guarantee that $a$ has a certain total type in the processed graph $\proc\GG$.

In preparation for this result, we recall some definitions and notations.
As in Definition~\ref{d:total.type}, we denote a clause total type as $\bL\equiv(L_0,L_\infty)$ where $L_0$ is the initial clause type and $L_\infty$ is the final clause type. The initial type $L_0$ corresponds to the $(R+1/2)$-neighborhood of the clause. If the clause does not lie in a compound enclosure (Definition~\ref{d:enclosure}), then the final type $L_\infty$ corresponds to the $(R+1/2)$-neighborhood of the clause in the processed graph. If the clause does lie in a compound enclosure $U$, then $L_\infty$ instead encodes the structure of $B_R(U)$. Recall also from Definition~\ref{d:enclosure} that any compound enclosure has diameter at most $R/100$, so in any case $L_\infty$ has depth at most $R(1+1/100)$. Of course, we restrict our attention to types $\bL$ that can actually occur, as formalized by the following:

\begin{dfn}[feasible types] \label{d:feasible.clause.type}A clause total type $\bL\equiv(L_0,L_\infty)$ is termed \bemph{feasible} if there exists some bipartite factor graph $\HH$, with girth larger than $8R$, such that $\P_{n',n'\alpha'}(\HH)>0$ for some $n',\alpha'$, and some clause $a_\star\in\proc\HH$ has total type $\bL$. Thus $L_0$ and $L_\infty$ are both trees, which we regard as being rooted at $a_\star$.\end{dfn}

Given a feasible type $\bL$, we now construct the tree $\tL$ (see Figure~\ref{f:T.construction} for a schematic depiction):

\begin{dfn}[tree based on feasible type]
\label{d:tree.based.on.feasible.type}
Given a feasible clause total type $\bL$, we now construct a corresponding tree $\tL$ as follows. Among all pairs $(\HH,a_\star)$ for which the conditions of Definition~\ref{d:feasible.clause.type} hold, fix one such that $\HH$ has minimal size. Let $\tree_\star$ be the $(4R)$-neighborhood of $a_\star$ in $\HH$, which we regard as a tree rooted at $a_\star$. We regard $L_0,L_\infty$ as subtrees of $\tree_\star$, and make the following definitions:
\begin{enumerate}[a.]
\item Let $V_1$ denote the set of variables of $L_0\setminus L_\infty$ which neighbor some clause in $L_\infty$. For each $v\in V_1$, let $\tree_\star(v,\ell)$ be the subtree of $\tree_\star$ induced by all the descendants of $v$ that lie within distance $\ell$ of $v$. Let $\Yout(v)$ be the variables in $\tree_\star(v)$ that lie at depth exactly $R$ below $v$ --- we will argue in the proof of Corollary~\ref{c:event.implies.desired.type} below that $\Yout(v)$ is nonempty for all $v\in V_1$. We take a union over $V_1$ to define
	\beq\label{e:defn.Tout}
	\tree_{\star,\textup{out}}
		\equiv\bigcup_{v\in V_1} \tree_\star(v,2R)\,,\quad
	\Nout\equiv\bigcup_{v\in V_1} \tree_\star(v,R-1)\,,\quad
	\Yout\equiv\bigcup_{v\in V_1} \Yout(v)\,.
	\eeq
For each $y\in\Yout$, let $\Zout(y)$ be the variables that lie at depth exactly $R$ below $y$. If $\Zout(y)$ is nonempty, then we fix an arbitrary $w\in\Zout(v)$, and modify $w$ by redefining $\LABEL(w)$ to be the same as $\LABEL(v)$ where $v$ is the ancestor of $y$ in $V_1$. We let $\Tout$ denote the resulting modification of $\tree_{\star,\textup{out}}$ (the two graphs can differ only in the markings $\LABEL$ at the last level).

\item Now let $\tree_{\star\star}$ be the component of $\tree_\star\setminus\Tout$ that contains the root clause $a_\star$. Let $W$ denote the variables at the boundary of $L_\infty$. From the definition of the processing algorithm, each variable left in $\proc\HH$ is $1$-good, which means (Definition~\ref{d:good.exc}) that any length-$(2R/5)$ path emanating from the variable must contain at least one variable that is $1$-excellent. This in turn means we can find a set $X$ of $1$-excellent variables that lie below $W$, at distance at most $2R/5$ from $W$, such that $X$ forms a cutset in $\tree_{\star\star}$ that encloses $L_\infty$. Let $Y$ be the variables in $\tree_{\star\star}$ that lie at depth exactly $R/10$ below $X$. Let $\Tin$ be the subtree of $\tree_{\star\star}$ enclosed by $Y$. 
\end{enumerate}
Let $\tL$ be the subtree of $\tree_\star$ induced by the union of $\Tin$ and $\Tout$.
\end{dfn}

\begin{figure}[h]
\includegraphics[scale=1]{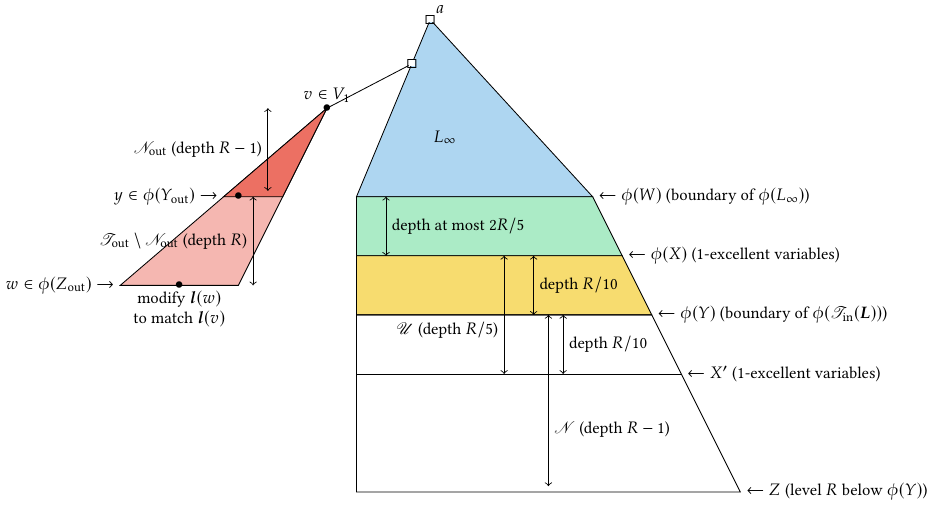}
\caption{Depiction of the events $\bm{E}_a(\bL)$ and $\bm{G}_a(\bL)$
(Definition~\ref{d:pos.frac.events}). The shaded region shows $\phi(\tree(\bL))$
where $\bL$ is a feasible type (Definition~\ref{d:feasible.clause.type}),
 $\tree(\bL)$ is the tree based on $\bL$ (Definition~\ref{d:tree.based.on.feasible.type}), and $\phi$ embeds $\tree(\bL)$ into a neighborhood of $a$ on the event $\bm{E}_a(\bL)$. The unshaded region points out the main features of the event $\bm{G}_a(\bL)$.}
\label{f:T.construction}
\end{figure}

We next define an event $\bm{E}_a$ which says, roughly, that the neighborhood of $a$ in $\GG$ looks like $\tL$ and does not contain cycles of length $\le 8R$. We then define events $\bm{G}_a$ and $\bm{K}_a$ which capture additional desirable properties.

\begin{dfn}[events $\bm{E}_a$, $\bm{G}_a$, $\bm{K}_a$]\label{d:pos.frac.events}
Fix a feasible type $\bL$, and let $\tL$ be as given by Definition~\ref{d:tree.based.on.feasible.type}, rooted at clause $a_\star$. Let $\GG=(V,F,E)$ be sampled from $\P_{n,m}$. For any clause $a\in F$, we let $\bm{E}_a\equiv\bm{E}_a(\bL)$ be the event that $\tree'\equiv B_{4R}(a;\GG)$ is a tree, and there is an embedding $\phi:\tL\hookrightarrow\tree'$ that maps $a_\star\mapsto a$, and satisfies the following:
\begin{enumerate}[(i)]
\item \label{Ea.i} For any variable $v\in \phi(\tL)\subseteq\tree'$, either all or none of its child variables (with respect to $\tree'$) lie in $\phi(\tL)$.

\item \label{Ea.ii} For any variable $v\in V_1$, the subtree of $v$ in $\tL$ agrees with the depth-$2R$ subtree of $\phi(v)$ in $\tree'$.
\end{enumerate}
If $\bm{E}_a$ occurs, let $X'$ be the cutset of variables in $\tree'$ lying at distance exactly $R/10$ below $\phi(Y)$. Let $Z$ be the cutset of
variables in $\tree'$ lying at distance exactly $R$ below $\phi(Y)$. Let $\UU$ be the subgraph of $\tree'$ that lies sandwiched between $\phi(X)$ and $X'$ (inclusive). Let $\NN$ be the subgraph of $\tree'$ that lies sandwiched between $\phi(Y)$ and $Z$, including $\phi(Y)$ but not including $Z$. Let $\bm{G}_a\equiv\bm{G}_a(\bL)$ denote the event that $\GG\in\bm{E}_a$, and moreover satisfies the following:
\begin{enumerate}[(I)]
\item \label{i:Tin.and.P.are.proper} All variables in $\phi(\Tin)\cup\NN$ are proper;
\item \label{i:Tin.Tout.diff.labels} $\LABEL(u)\ne\LABEL(w)$
for any $u\in\phi(\Tout)$ and $w\in\phi(\Tin)\cup\NN$; 
\item \label{i:U.is.fair} All variables in $\UU$ are $1$-fair;
\item \label{i:Xprime.is.excellent} All variables in $X'$ are $1$-excellent.
\end{enumerate}
If $\bm{E}_a$ occurs, then we also define the subgraph
	\beq\label{e:Gprime.cavity}
	\cGG\equiv \GG \,\bigg\backslash\,
		\phi\Big( \Tin \setminus Y\Big)\,.\eeq
We then let $\bm{K}_a\equiv\bm{K}_a(\bL)$ be the event that $\GG\in\bm{E}_a$, and every variable in $\NN\cup Z$ survives in $\proc\cGG$.
\end{dfn}

\begin{lem}\label{l:event.implies.A.Aprime.comparison} Suppose that $\bL$ is a feasible clause type in the sense of Definition~\ref{d:feasible.clause.type}. Suppose $\GG$ has girth greater than $8R$ and belongs to the event $\bm{G}_a \cap \bm{K}_a$, and let $\cGG$ be as defined by \eqref{e:Gprime.cavity}. Let $A,A'$ be the variables that are improper or not $1$-good with respect to $\GG,\cGG$ respectively. Then $A\subseteq A' \subseteq \cGG\setminus\NN$
and $A'\setminus A= \phi(\Nout) \cap A'$.

\begin{proof}Throughout the proof we fix $\HH$ as in Definition~\ref{d:feasible.clause.type}, and let $\tL$ be the tree given by Definition~\ref{d:tree.based.on.feasible.type}. Let $A_{\HH}$ denote the set of all variables that are improper or not $1$-good with respect to $\HH$. Then, with $\Nout$ as defined by \eqref{e:defn.Tout}, we first observe that
	\beq\label{e:obs.A.HH}
	A_{\HH} \cap \Big( \Tin \cup \Nout \Big) = \emptyset
	\eeq
--- this is because the $R$-neighborhood of each variable in $A_{\HH}$ is removed in the first step of processing on $\HH$, but $\Tin$ is assumed to survive in $\proc\HH$.

We now turn to the comparison of $A$ and $A'$. To begin, recall that one way for a variable to be improper is that its $R$-neighborhood contains a cycle --- however, by the girth assumption, this never happens in $\GG$ or $\cGG$. It remains to consider all the ways for an acyclic variable to belong to $A$ or $A'$:
\begin{enumerate}[$\bullet$]
\item Let $A(R)$ denote the variables $v\in A$ for which $B_R(v;\GG)$ contains a repeat marking --- i.e., two variables $u\ne w$ with the same marking $\LABEL(u)=\LABEL(w)$. Define similarly $A'(R)\subseteq A'$, and note $A'(R)\subseteq A(R)$ since the $R$-neighborhood of any variable relative to $\cGG$ is a subgraph of its $R$-neighborhood relative to $\GG$. Let 
	\[\GGR\equiv \cGG\,\Big\backslash\,
		\Big( \NN \cup\phi(\Nout)\Big)\,,\]
that is to say, $\GGR$ is the subgraph of $\GG$ induced by variables lying at distance at least $R$ from $\phi(\Tin)$. If variable $v$ lies in $\GGR$, then $B_R(v;\GG)=B_R(v;\cGG)$, and so
	\[
	A(R) \cap \GGR = A'(R) \cap \GGR\,.
	\]
If $v$ lies in $\GG\setminus\GGR$, we distinguish two cases:
\begin{enumerate}[$\circ$]
\item If $v\in\phi(\Tin)\cup\NN$,
property~\eqref{i:Tin.and.P.are.proper} implies that $v$ must be proper in $\GG$, so $v\notin A(R)$. 

\item If $v\in\phi(\Nout)$, it follows from \eqref{e:obs.A.HH} that
$v'\equiv\phi^{-1}(v)$ must be proper with respect to $\HH$, so $B_R(v';\HH)$ has no repeat marking. We claim that the same holds for $B_R(v;\GG)$. Indeed, suppose for contradiction that
in $B_R(v;\GG)$ there are two variables $u\ne w$ with $\LABEL(u)=\LABEL(w)$.
\begin{enumerate}[--]
\item If both $u,w$ lie in $\phi(\Tout)$, then we have two variables
$\phi^{-1}(u)\ne\phi^{-1}(w)$ inside $B_R(v';\HH)$ with the same marking, contradicting the above observation.
\item If both $u,w$ lie in $\phi(\Tin)\cup\NN$,
then it must be possible to join them by a path
\emph{inside} $\phi(\Tin)\cup\NN$ of length at most $2R-1$.
It follows that $u,w\in B_R(\bar{v};\GG)$ for $\bar{v}\in\phi(\Tin)\cup\NN$,
which contradicts property~\eqref{i:Tin.and.P.are.proper}.
\item If $u\in\phi(\Tout)$ while $w\in\phi(\Tin)\cup\NN$, we must have $\LABEL(u)\ne\LABEL(w)$ by property~\eqref{i:Tin.Tout.diff.labels}. 
\end{enumerate}
\end{enumerate}
The above shows that $A(R)$ cannot intersect $\GG\setminus\GGR$.
Since $A'(R)\subseteq A(R)$, we obtain that $A'(R)$
also cannot intersect $\GG\setminus\GGR$. 
Combining with the earlier observation gives
	\beq\label{e:matching.A.R.sets}
	A(R)=A'(R)\subseteq \GGR\,.
	\eeq

\item Let $A(1),A'(1)$ denote the variables that fail to be $1$-good with respect to $\GG,\cGG$ respectively. Recall that any variable in $\GG$ or $\cGG$ is acyclic, so whether it is $1$-good depends only on its 
 $(2R/5)$-neighborhood. If $v$ lies in $\GG$ at distance at least $2R/5$ from $\phi(\Tin)$, then it has the same $(2R/5)$-neighborhood in both $\GG$ and $\cGG$, so $v\in A(1)$ if and only if $v\in A'(1)$. For $v$ lying at distance less than $2R/5$ from $\phi(\Tin)$, we distinguish three cases:
\begin{enumerate}[$\circ$]
\item First suppose that $v$ lies either in $\phi(\Nout)$, or in the part of $\phi(\Tin)$ above $\phi(X)$. Then it follows from \eqref{e:obs.A.HH} that $v'\equiv\phi^{-1}(v)$ must be $1$-good (hence $1$-fair) in $\HH$. Since being $1$-fair is a property of the $(R/10)$-neighborhood, and $B_{R/10}(v;\GG)\cong B_{R/10}(v';\HH)$, we see that $v$ is $1$-fair in $\GG$. Now consider a path $\gamma$ of length $2R/5$ emanating from $v$:
\begin{enumerate}[--]
\item Suppose $\gamma$ never intersects $\phi(X)$. Since $v'=\phi^{-1}(v)$ is $1$-good in $\HH$, the path $\phi^{-1}(\gamma)$ must contain a variable $u'=\phi^{-1}(u)$ which is $1$-excellent in $\HH$. Since being $1$-excellent is a property of the $(R/10)$-neighborhood, and $B_{R/10}(u;\GG)\cong B_{R/10}(u';\HH)$, we see that $u\in\gamma$ is $1$-excellent in $\GG$.

\item Otherwise, $\gamma$ contains a variable $u\in\phi(X)$.
From the construction (Definition~\ref{d:tree.based.on.feasible.type}),
the variable $u'=\phi^{-1}(u)$ is $1$-excellent in $\HH$. By the same reasoning as in the last case, $u$ must then be $1$-excellent in $\GG$.
\end{enumerate}
This proves that $v$ is $1$-good in $\GG$, i.e., $v\notin A(1)$.

\item Next suppose $v$ lies in $\UU$ (i.e., between $\phi(X)$ and $X'$, inclusive). Then $v$ is $1$-fair in $\GG$ simply by property~\eqref{i:U.is.fair}. A path of length $2R/5$ emanating from $v$ must contain a variable $u$ from either $\phi(X)$ or $X'$.
If $u\in\phi(X)$, then it is $1$-excellent in $\GG$ as argued above. If $u\in X'$, then it is $1$-excellent in $\GG$ by property~\eqref{i:Xprime.is.excellent}. Thus, $v\notin A(1)$.

\item It remains to consider the case that $v$ lies in $\NN\setminus\UU$.
For any variable in $\GG$ at distance at least $R/10$ from $\phi(\Tin)$,
the $(R/10)$-neighborhoods in $\GG,\cGG$ are the same,
so that variable is $1$-fair in $\GG$ if and only if it is $1$-fair in $\cGG$;
the same applies to the $1$-excellent property.
In particular, this tells us that $v\in\NN\setminus\UU$ is $1$-fair in $\GG$ if and only if it is $1$-fair in $\cGG$. If a path of length $2R/5$ emanates from $v$ and never intersects $X'$, then it must stay at least distance $R/10$ away from $\phi(\Tin)$, so the path contains a $1$-excellent variable of $\GG$ if and only if it contains a $1$-excellent variable of $\cGG$. Otherwise, the path contains a variable from $X'$ which is $1$-excellent in both $\GG$ and $\cGG$ by property~\eqref{i:Xprime.is.excellent}. This proves that
	$A(1) \cap (\NN\setminus\UU) = A'(1) \cap (\NN\setminus\UU)$.
\end{enumerate}
It follows from the above that $A(1)\subseteq A'(1)$ and
	\beq\label{e:compare.A.one.sets}
	A(1) \cap \phi(\Nout) = \emptyset\,.
	\eeq
\end{enumerate}
Now note that $A'$ is by definition a subset of $\cGG$, and on the event $\bm{K}_a$ it must not intersect $\NN$. Thus, on the event $\bm{K}_a$, we have $A(1)\subseteq A'(1)\subseteq\cGG\setminus\NN$. We also noted that variables at distance at least $2R/5$ from $\phi(\Tin)$ belong in $A(1)$ if and only if they belong in $A'(1)$, so in particular $A(1)\cap\GGR=A'(1)\cap\GGR$. It follows that
	\[A'(1)\setminus A(1)
	\subseteq(\cGG\setminus\NN)
		\Big\backslash \GGR
	=\phi(\Nout)\,.\]
We already saw in \eqref{e:matching.A.R.sets} (without appealing to $\bm{K}_a$) that $A(R)=A'(R)\subseteq\GGR$. Combining with \eqref{e:compare.A.one.sets} finishes the proof of the lemma, since $A=A(R)\cup A(1)$ and $A'=A'(R)\cup A'(1)$.
\end{proof}
\end{lem}

\begin{cor}\label{c:event.implies.desired.type} Suppose that $\bL\equiv(L_0,L_\infty)$ is a feasible clause type in the sense of Definition~\ref{d:feasible.clause.type}. Let $\bm{G}_a,\bm{K}_a$ be the events from Definition~\ref{d:pos.frac.events}. If $\GG$ has girth greater than $8R$ and belongs to the event $\bm{G}_a \cap \bm{K}_a$, then the clause $a$ has total type $\bL_a=\bL$.


\begin{proof} As before, we fix $\HH$ as in Definition~\ref{d:feasible.clause.type}, and let $\tL$ be the tree given by Definition~\ref{d:tree.based.on.feasible.type}. We first argue that (using the same notation as in Definition~\ref{d:tree.based.on.feasible.type}) we have
	\beq\label{e:Yout.intersect.A.nonempty}
	\Yout(v) \cap A \ne\emptyset\quad\textup{for all }v\in V_1\,.
	\eeq
We separate this into two cases:
\begin{enumerate}[$\bullet$]
\item If $\Zout(y)\ne\emptyset$ for some $y\in\Yout(v)$, then recall from Definition~\ref{d:tree.based.on.feasible.type} that we choose some $w\in\Zout(y)$ and set $\LABEL(w)=\LABEL(y)$, which ensures $y\in A(R)\subseteq A$.

\item Now suppose that we do not have $\Zout(y)\ne\emptyset$ for any $y\in\Yout(v)$. Let $a'$ be the parent clause of $v$ in $\tL$, so $a'$ lies in $\Tin$. Since the preprocessing algorithm on $\HH$ removes $v$ but leaves $a'$ behind, it must be that at some stage of the algorithm a removal is triggered by a vertex $u$ which lies at depth exactly $R$ below $v$. (In particular, $u\in\Yout(v)$ which proves that $\Yout(v)$ must be nonempty.) We now argue that in fact there must be some $u\in\Yout(v)\cap A$. If not, then $u$ triggers a removal \emph{after} the initial stage of $\bsp'(A;\GG)$ --- that is to say, at some stage, the $R$-neighborhood of $u$ must contain either one clause of degree $\le k-2$, or two clauses of degree $k-1$. Let $u'$ be the parent variable of $u$: by the assumption that $\Zout(v)=\emptyset$, the $R$-neighborhood of $u'$ contains the $R$-neighborhood of $u$, which means that $u'$ should trigger a removal at the same stage of $\bsp'(A;\GG)$. This is a contradiction, since removing the $R$-neighborhood of $u'$ would remove the parent clause $a'$ of $v$.
\end{enumerate}
This concludes the proof of \eqref{e:Yout.intersect.A.nonempty}, which immediately implies that $V_1\subseteq B_R(A';\cGG)$. For the remainder of the proof, we partition $\Yout=Y_+\sqcup Y_\emptyset$ where
	\[Y_+\equiv\bigg\{ y\in \Yout : \Zout(y)\ne\emptyset\bigg\}\,.
	\]
Note that, by the construction from Definition~\ref{d:tree.based.on.feasible.type}, $Y_+\subseteq A(R)$.
Let $\tree_\emptyset$ denote the subgraph of $\Tout$ induced by all descendants of variables in $Y_\emptyset$.

Now assume that $\GG$ has girth greater than $8R$ and belongs to the event $\bm{G}_a \cap \bm{K}_a$. This means that the clause $a\in\GG$ has initial type $L_0$. We next argue that
	\beq\label{e:same.initial.removal.outside.Nout}
	B_R(A';\cGG) \Big\backslash \phi(\Nout)
	= \bigg\{
	B_R(A;\GG) \Big\backslash \phi(\Nout)
	\bigg\} \sqcup K_\emptyset
	\eeq
for some $K_\emptyset \subseteq \tree_\emptyset$. Indeed, recall that if $v\in\GGR$ then $v$ has the same $R$-neighborhood in $\GG$ as in $\cGG$, and 
Lemma~\ref{l:event.implies.A.Aprime.comparison}
 implies that $A\subseteq \GGR$, $A\cap\GGR=A'\cap\GGR$,
and $A'\setminus A =\phi(\Nout) \cap A'$. As a result we can express
	\begin{align*}
	B_R(A';\cGG) \Big\backslash \phi(\Nout)
	&=\bigg\{ \bigcup_{v\in A} \Big(B_R(v;\GG) \Big\backslash \phi(\Nout)\Big)\bigg\}
	\cup\bigg\{ \bigcup_{v\in A'\setminus A}
		\Big( B_R(v;\cGG) \Big\backslash 
	\phi(\Nout)\Big)\bigg\}\\
	&=\bigg\{B_R(A;\GG) \Big\backslash \phi(\Nout)\bigg\}
	\cup\bigg\{
	\bigcup_{v\in \phi(\Nout)\cap A'}\Big(B_R(v;\cGG) \Big\backslash 
	\phi(\Nout)\Big)\bigg\}\,.
	\end{align*}
Now consider $v\in \phi(\Nout)\cap A'$ 
and $u\in B_R(v;\cGG)\setminus\phi(\Nout)$.
The path between $u$ and $v$ must intersect exactly one variable $y\in\Yout$. If this $y$ belongs to $Y_+$, then (since $Y_+\subseteq A$) we have
	\[
	u \in B_R(y;\GG)\Big\backslash\phi(\Nout)
	\subseteq B_R(A;\GG)\Big\backslash\phi(\Nout)\,.
	\]
If instead $y$ belongs to $Y_\emptyset$, then $u\in K_\emptyset$. Combining these observations proves \eqref{e:same.initial.removal.outside.Nout}.

Next, we also observe that $B_R(\Yout;\GG) \cap \phi(\Nout)$ is the same as $B_R(A;\GG) \cap \phi(\Nout)$, and this must be a subset of $B_R(A';\cGG) \cap \phi(\Nout)$. 
Denote
	\[
	K \equiv \phi(\Nout)\Big\backslash B_R(A;\GG)
		= \phi(\Nout)\Big\backslash B_R(\Yout;\GG)
		\,,\quad
	K'\equiv \phi(\Nout)\Big\backslash B_R(A';\cGG)
		\subseteq K\,.
	\]
It follows using \eqref{e:same.initial.removal.outside.Nout} that
	\begin{align}\nonumber
	\cGG\Big\backslash B_R(A;\GG)
	&=\Big(\GGR \cup \NN \cup \phi(\Nout)\Big)
	\Big\backslash B_R(A;\GG)
	=\bigg\{\Big(\GGR \cup \NN\Big)
	\Big\backslash B_R(A';\cGG)\bigg\}
	\cup (K \cup K_\emptyset) \,,\\ \label{e:Gprime.with.Kprime}
	\cGG(0) \equiv \cGG\Big\backslash B_R(A';\cGG)
	&=\Big(\GGR \cup \NN \cup \phi(\Nout)\Big)
	\Big\backslash B_R(A';\cGG)
	=\bigg\{\Big(\GGR \cup \NN\Big)
	\Big\backslash B_R(A';\cGG)\bigg\}
	\cup K'\,,\end{align}
where $\cGG(0)$ is what is left of $\cGG$ after the initial stage of $\bsp'(A';\cGG)$. For comparison, if $\GG(0)$ denotes what is left of $\GG$ after the initial stage of $\bsp'(A;\GG)$, then
	\beq\label{e:G.with.K}
	\GG(0) \equiv \GG \Big\backslash B_R(A;\GG)
	=\bigg\{\Big(\GGR \cup \NN\Big)
	\Big\backslash B_R(A';\cGG)\bigg\} \cup \phi(\Tin\setminus Y)
	\cup (K \cup K_\emptyset)\,.
	\eeq
Recall that \eqref{e:Yout.intersect.A.nonempty} implies that $V_1\subseteq B_R(A';\cGG)$. Consequently, in \eqref{e:Gprime.with.Kprime} the subgraph induced by $K'$ is disconnected from the rest, and likewise in \eqref{e:G.with.K} the subgraph induced by $K\cup K_\emptyset$ is disconnected from the rest. It follows by comparing \eqref{e:Gprime.with.Kprime} with \eqref{e:G.with.K} that 
	\beq\label{e:coupling.G.Gprime.zero}
	\GG(0) \Big\backslash (K\cup K_\emptyset)
	= \Big( \cGG(0) \cup \phi(\Tin\setminus Y)\Big)
		\Big\backslash K'\,.
	\eeq
Let $\GG(t)$ and $\cGG(t)$ be the graphs $\GG,\cGG$ after the $t$-th stages of $\bsp'(A;\GG)$ and $\bsp'(A';\cGG)$ respectively. We will argue by induction (with \eqref{e:coupling.G.Gprime.zero} being the base case) that
	\beq\label{e:coupling.G.Gprime}
	\GG(t) \Big\backslash (K\cup K_\emptyset)
	= \Big( \cGG(t) \cup \phi(\Tin\setminus Y)\Big)
		\Big\backslash K'\,.
	\eeq
Indeed, suppose inductively that \eqref{e:coupling.G.Gprime} holds up to stage $t\ge0$. Recall the notation \eqref{e:bspprime.activated.set}. It is clear that
	\[A'(t)\equiv\ACT\Big(\cGG(t) \setminus K'\Big)
	\subseteq
	\ACT\Big(\GG(t) \setminus K\Big) \equiv A(t)\,,
	\]
where we have discarded the disconnected components induced by
 $K\cup K_\emptyset$ and by $K'$
 in defining $A(t),A'(t)$. In order for $v$ to belong to $A(t)$, there are two possibilities:
\begin{enumerate}[$\bullet$]
\item In $\GG(t)$ there is a path $\gamma$ of length at most $3R/10$ that joins $v$ to a clause $a'$ of degree $\le k-2$. If $\gamma$ is contained in $\cGG(t)$, then clearly $v$ must belong to $A'(t)$ as well. Otherwise, if $\gamma$ is not contained in $\cGG(t)$, then using the inductive hypothesis \eqref{e:coupling.G.Gprime} it must be that $\gamma$ intersects $\phi(\Tin\setminus Y)$. On the other hand, the clause $a'$ cannot be in $\phi(\Tin\setminus Y)$ (which cannot contain any clauses of degree $\le k-2$, since it was constructed from a processed graph $\proc\HH$). Moreover, on the event $\bm{K}_a$, the clause $a'$ cannot be in $\NN$. This is a contradiction, since there is no path of length $3R/10$ that intersects both $\phi(\Tin\setminus Y)$ and $\GG(t)\setminus[\phi(\Tin\setminus Y) \cup\NN]$.

\item In $\GG(t)$ there are two paths $\gamma_1,\gamma_2$ of length at most $3R/10$ that join $v$ to clauses $a_1\ne a_2$ of degree $k-1$.
If the $\gamma_i$ are both contained in $\cGG(t)$, then $v$ must belong to $A'(t)$ as well. Otherwise, by the same argument as above, at least one of the paths must intersect $\phi(\Tin\setminus Y)$. On the event $\bm{K}_a$, the clauses $a_i$ cannot be in $\NN$, so they must both be $\phi(\Tin\setminus Y)$. This contradicts the construction of $\phi(\Tin\setminus Y)$ which was based on the processed graph $\proc\HH$.
\end{enumerate}
The above shows that $A(t)=A'(t)\subseteq\cGG(t)\setminus\NN$,
from which it follows that
	\[B_R\Big(A(t);\GG(t)\Big)
	=B_R\Big(A'(t);\cGG(t)\Big)\subseteq \cGG(t)\,.\]
This implies that \eqref{e:coupling.G.Gprime} holds at the next stage $t+1$. It follows that
	\[
	\proc\GG\Big\backslash (K\cup K_\emptyset)
	= \Big(\proc\cGG \cup \phi(\Tin\setminus Y)\Big)
		\Big\backslash K'\,,
	\]
so that the clause $a$ in graph $\GG$ has final type $L_\infty$.
It follows that clause $a$ has total type $\bL_a=(L_0,L_\infty)$, as required.
\end{proof}
\end{cor}

\subsection{Probabilistic analysis for positive type fractions}\label{ss:positive.fraction.prob}

We now conclude the proof of Proposition~\ref{p:posfrac}. In view of Corollary~\ref{c:event.implies.desired.type} from the preceding subsection, it suffices to show that for any feasible $\bL$, conditional on $\GG$ having girth greater than $8R$, with high probability the events $\bm{G}_a(\bL)$ and $\bm{K}_a(\bL)$ (Definition~\ref{d:pos.frac.events}) will occur for a positive fraction of clauses $a\in F$. We argue this in a few steps, below. The general idea is that all these events are fairly local in nature, so they should occur a linear number of times in the random graph $\GG$.

\begin{lem}\label{l:lbd.Ea}
Let $\P=\P_{n,m}$ for $|m-n\alpha| \le n^{1/2}\log n$, and let $\GG$ denote a sample from $\P$. There is a positive constant $c_0(k,R)$ such that for all $n$ large enough, we have
	\[
	\P\Big( \GG\in \bm{E}_a(\bL)\Big) \ge c_0(k,R)\,.
	\]
for all $a\in[m]$ and all clause total types $\bL$ that are feasible in the sense of Definition~\ref{d:feasible.clause.type}.

\begin{proof} We first argue that there is a finite constant $C(k,R)$ such that
	\beq\label{e:frak.L.bounded}
	\bigg|\bigg\{
	\hspace{-3pt}\begin{array}{c}
	\textup{feasible clause total} \\ 
	\textup{types $\bL\equiv(L_0,L_\infty)$}
	\end{array}\hspace{-3pt}
	\bigg\}\bigg| \le C_0(k,R)\,.
	\eeq
To this end, let $\bL\equiv(L_0,L_\infty)$ be a feasible type, with $\HH,a_\star$ as in Definition~\ref{d:feasible.clause.type}. Recall that the initial type $L_0$ encodes the $(R+1/2)$-neighborhood of $a_\star$ in $\HH$. If any variable in the $(R-1/2)$-neighborhood of $a_\star$ in $\HH$ has degree more than $\exp\{k^2 (5\rprime)\}$, then one of the variables $u\in a_\star$ will fail to be $1$-fair in $\HH$ (since it will violate property \eqref{i:fair.volume.bound} in Definition~\ref{d:perfect.fair}). But then the $R$-neighborhood of $u$ will be removed in the initial stage of processing on $\HH$, contradicting the assumption that $a_\star\in\proc\HH$. This proves that all variables in $L_0$ must have degree at most
$\exp\{k^2 (5\rprime)\}$. All variables in the final type $L_\infty$ must be $1$-good, hence also $1$-fair, so they must also have degree at most
$\exp\{k^2 (5\rprime)\}$. Finally $L_0$ is a tree of depth at most $R$, while $L_\infty$ is a tree of depth at most $R(1+1/100)$. This proves \eqref{e:frak.L.bounded}. Since we fix a mapping from $\bL$ to $\tL$ (see Definition~\ref{d:tree.based.on.feasible.type}), it immediately follows from \eqref{e:frak.L.bounded} that the number of distinct $\tL$ is also upper bounded by $C_0(k,R)$.

Now take $\P=\P_{n,m}$ as in the statement of the lemma, fix $a\in[m]$, and let $\GG\sim\P$. By revealing the neighborhood of of $a$ in $\GG$ in breadth-first fashion, it is easy to see that $\P(\bm{E}_a(\bL) )$ is lower bounded by a constant which depends only on $\tL$. It then follows from the above that in fact there is a constant $c_0(k,R)$ such that (for $n$ large enough) we have $\P(\bm{E}_a(\bL) ) \ge c_0(k,R)$ for all feasible $\bL$. This proves the lemma.
\end{proof}
\end{lem}

\begin{lem}\label{l:prob.Ga.given.Ea}
Let $\P=\P_{n,m}$ for $|m-n\alpha| \le n^{1/2}\log n$, and $\GG\sim\P$. Then, for all $n$ large enough, we have
	\[
	\P\Big(\GG\in\bm{G}_a(\bL) \,\Big|\, 
	\GG\in\bm{E}_a(\bL)\Big) \ge \f78\,.
	\]
for all $a\in[m]$ and all clause total types $\bL$ that are feasible in the sense of Definition~\ref{d:feasible.clause.type}.

\begin{proof}
Throughout the proof, $\bL$ is fixed and often suppressed from the notation. Let $\bm{P}_a$ be the event that $\bm{E}_a$ holds, and that $\GG$ satisfies properties \eqref{i:Tin.and.P.are.proper} and \eqref{i:Tin.Tout.diff.labels} of Definition~\ref{d:pos.frac.events}. Then $\bm{E}_a\subseteq\bm{P}_a\subseteq\bm{K}_a$, and it is clear that
	\beq\label{e:Pa.bound}
	\P\Big(\GG\in\bm{P}_a \,\Big|\,\GG\in \bm{E}_a
	\Big) \ge \f{24}{25}\,.
	\eeq
This is simply because the markings $\LABEL$ are chosen uniformly at random from a very large set (see \eqref{e:random.R.marking}), so if we consider all the variables in $\GG\setminus\phi(\tL)$ within distance $2R$ of $\phi(\tL)$, it holds with probability at least $9/10$ that all their markings are distinct from one another and from the markings on $\phi(\tL)$.

Next we let $\bm{F}_a$ be the event that $\bm{E}_a$ holds, and that 
$\GG$ satisfies property \eqref{i:U.is.fair} of Definition~\ref{d:pos.frac.events}.
For $x\in\phi(X)$, let $\UU(x)$ denote the subtree of $\UU$ descended from $x$. Let $\bm{F}_a(x)$ be the event that $\bm{E}_a$ holds, and that every $u\in\UU(x)$ is $1$-fair (with respect to $\GG$). Then
	\[
	\bm{F}_a = \bigcap_{x\in\phi(X)} \bm{F}_a(x)\,.
	\]
For $x\in\phi(X)$, let $\UU_+(x)$ denote the subtree of $\UU\cup\NN$ induced by descendants of $x$ that lie within distance $3R/10$ of $x$, so that $\UU(x)\subseteq\UU_+(x)\subseteq\UU\cup\NN$. Since $\bm{F}_a\subseteq\bm{E}_a$ where $\bm{E}_a$ imposes that $B_{4R}(a;\GG)$ is a tree, for any variable in $B_{(4-1/10)R}(a;\GG)$ we can determine whether it is $1$-fair based on its $(R/10)$-neighborhood only. It follows that the event $\bm{F}_a(x)$ can be determined from
	\[
	B_{R/10}(x;\GG) \cup \UU_+(x)\,.
	\]
An important point is that the above does not see any part of $\phi(\tL)$ at distance more than $R/10$ from $x$, because
	\beq\label{e:conditioning.trick}
	\phi(\tL) \cap \bigg\{ B_{R/10}(x;\GG) \cup \UU_+(x)\bigg\}
		=\phi(\tL) \cap B_{R/10}(x;\GG)\,.\eeq
Indeed, when we condition on $\bm{E}_a$, we in fact reveal information about the neighborhood of $x$ in $\GG$ beyond depth $R/10$ (simply because $\phi(\tL)$ contains vertices at distance more than $R/10$ from $x$), but \eqref{e:conditioning.trick} allows us to disregard the additional information. It follows from the $1$-excellence condition \eqref{e:j.excellent} that
	\beq\label{e:Fa.x.lbd}
	\P\Big(\GG\in\bm{F}_a(x)\,\Big|\, \GG\in\bm{E}_a\Big)
		\ge 1- \f1{\exp(k^3 R)} + o_n(1)\,,\eeq
where the $o_n(1)$ comes from the discrepancy between the Galton--Watson law and the breadth-first exploration in $\GG$. Since all variables in $\Tin$ must be $1$-fair, it follows from property \eqref{i:fair.volume.bound} in Definition~\ref{d:perfect.fair} that $\Tin$ has cardinality at most $\exp(O(k^2 R))$. Therefore we conclude
	\beq\label{e:Fa.combined.bound}
	\P\Big(\GG\in\bm{F}_a\,\Big|\, \GG\in\bm{E}_a\Big) \ge \f{24}{25}\eeq
by taking a union bound of \eqref{e:Fa.x.lbd} over $x\in\phi(X)$.

Finally, let $\bm{D}_a$ be the event that $\bm{E}_a$ holds, and that $\GG$ satisfies property \eqref{i:Xprime.is.excellent} of Definition~\ref{d:pos.frac.events}. For $y\in\phi(Y)$, let $X'(y)$ denote the subset of variables in $X'$ that lie at depth $R/10$ below $y$. 
Let $X''(y)$ denote the variables in $X'(y)$ that fail to be $1$-excellent in $\GG$. Again, on the event $\bm{E}_a$ which imposes that $B_{4R}(a;\GG)$ is acyclic, for any variable in $B_{(4-1/10)R}(a;\GG)$ we can determine whether it is $1$-excellent based on its $(R/10)$-neighborhood only. Consequently, by similar considerations as for \eqref{e:Fa.x.lbd}, we have
	\begin{align}\nonumber
	\E \Big( X'(y)\,\Big|\, \GG\in\bm{E}_a\Big)
	&= o_n(1)
	+ \int 
	\bigg|\bigg\{\hspace{-3pt}\begin{array}{c}x\in\tree :
	\textup{$d(\vrt,x)=R/10$ }\\
	\textup{and $x$ is not $1$-excellent}
	\end{array}\hspace{-3pt}
	\bigg\}\bigg|
	\,d\uPGW(\tree)\\
	&= o_n(1)
	+ \int \mathbf{1}
	\bigg\{\hspace{-3pt}\begin{array}{c}
	\textup{the root $\vrt$ of $\tree$ is}\\
	\textup{not $1$-excellent}
	\end{array}\hspace{-3pt}
	\bigg\} Z_{R/10}(\tree)
	\,d\uPGW(\tree)\,,\label{e:before.cs.bound}
	\end{align}
where $Z_{R/10}(\tree)$ denotes the number of variables at depth $R/10$ in $\tree$, and the last equality is by the unimodularity property 
\eqref{e:unimodular}. Now recall from Corollary~\ref{c:excellent}
that
	\[\uPGW\Big(\textup{$\vrt$
	 not $1$-excellent}\Big)
	\le\f1{ \exp( 2^{k\DELTACONST/4}R) }\,.
	\]
On the other hand, by a similar argument as in the proof of 
Lemma~\ref{l:radon.derivative.pgwT}, we can bound
	\[\int \Big(Z_{R/10}(\tree)\Big)^2
	\,d\uPGW(\tree)
	\le \exp(O(kR))\,.
	\]
Combining these and applying Cauchy--Schwarz in \eqref{e:before.cs.bound} gives
	\[\E \Big( X'(y)\,\Big|\, \GG\in\bm{E}_a\Big)
	\le \f{ \exp(O(kR))}{ \exp( 2^{k\DELTACONST/4}R) }\,.
	\]
Since we noted above that $\Tin$ has cardinality at most $\exp(O(k^2 R))$, it follows by Markov's inequality and a union bound over $y\in\phi(Y)$ that
	\beq\label{e:Da.combined.bound}
	\P\Big((\bm{D}_a)^c \,\Big|\, \GG\in\bm{E}_a\Big)
	\le\sum_{y\in\phi(Y)}\E\Big( X'(y)\, \,\Big|\, \GG\in\bm{E}_a\Big)
	\le \f1{25}\,.\eeq
The lemma follows by combining \eqref{e:Pa.bound}, \eqref{e:Fa.combined.bound}, and \eqref{e:Da.combined.bound},
since $\bm{G}_a= \bm{P}_a \cap \bm{F}_a \cap \bm{D}_a$.
\end{proof}
\end{lem}

\begin{lem}\label{l:prob.Ka.given.Ea}
Let $\P=\P_{n,m}$ for $|m-n\alpha| \le n^{1/2}\log n$, and $\GG\sim\P$. Then, for all $n$ large enough, we have
	\[
	\P\Big(\bm{K}_a(\bL) \,\Big|\,\GG\in\bm{E}_a(\bL)\Big) 
	\ge \f{7}{8}\,.
	\]
for all $a\in[m]$ and all clause total types $\bL$ that are feasible in the sense of Definition~\ref{d:feasible.clause.type}.

\begin{proof} Again, throughout the proof $\bL$ is fixed and often suppressed from the notation. Let $\GG\sim\P=\P_{n,m}$. Recalling Definition~\ref{d:pos.frac.events}, let $\bm{C}_a$ be the event that there is an embedding $\phi:\tL\hookrightarrow\GG$ satisfying properties \eqref{Ea.i}~and~\eqref{Ea.ii} of Definition~\ref{d:pos.frac.events}. Then $\bm{E}_a$ is $\bm{C}_a$ with the additional restriction that $B_{4R}(a;\GG)$ is a tree. Provided that $\GG\in\bm{C}_a$, we define, as in \eqref{e:Gprime.cavity},
	\[\cGG\equiv \GG \,\bigg\backslash\,
		\phi\Big( \Tin \setminus Y\Big)\,.\]
Recalling Definition~\ref{d:pos.frac.events}, we can express $\bm{K}_a$ as the intersection of $\bm{E}_a$ with the event that no variable in $\phi(Y)$ lies within distance $R$ of any variable removed in the processing of $\cGG$, equivalently,
	\beq\label{e:equiv.defn.Ka}
	\bm{K}_a
	= \bm{E}_a\cap \bigg\{
	\textup{$\phi(Y)$ does not intersect
		$S(\cGG)\equiv B_R\Big(\bsp'(A';\cGG);\cGG\Big)$}
	\bigg\}\,.
	\eeq
Now consider the graph $\cGG$ where only $\phi(\Tout)$ is labelled (in particular, we ignore $\phi(Y)$ for the moment). Let $n'$ denote $n$ minus the number of internal variables in $\Tin\setminus Y$. Likewise, let $m'$ denote $m$ minus the number of clauses in $\Tin\setminus Y$. Under the measure $\P_{n,m}(\cdot\,|\,\GG\in\bm{C}_a)$, the induced law of $\cGG$ is equivalent (up to graph isomorphism) to the law $\P_{n',m'}(\cdot\,|\,\cGG\in\bm{C}')$ where $\bm{C}'$ is the event that there is an (arbitrary, fixed) embedding $\phi_\textup{out}:\Tout\hookrightarrow\cGG$. This is only a minor modification of the original measure $\P_{n,m}$, and a trivial extension of Proposition~\ref{p:small.fraction.removed.in.processing} gives
	\beq\label{e:small.removal.without.girth.cond}
	\P_{n',m'}\bigg(|S(\cGG)| \ge \f{n'}{\exp(2^{ck}R)}
	\,\bigg|\,\cGG\in \bm{C}'\bigg) = o_n(1)\,.\eeq
Since we can choose a fixed embedding $\phi_\textup{out}$, we will suppose that it maps the variables in $\Tout$ to the last variables in $\cGG$, so that the variables left in $\cGG\setminus\phi_\textup{out}(\Tout)$ can be written as $[n'']=\set{1,\ldots,n''}$. We now return to $Y'\equiv\phi(Y)$, which we regard as an element of
\[
	[n'']_y
	=\Big\{(v_1,\ldots,v_y)\in[n'']^y
	: \textup{$v_i \ne v_j$ for all $i\ne j$}\Big\}\,,
	\]
where $y\equiv|Y|$. Given $(\cGG,Y')$ where $\cGG\in\bm{C}'$ and $Y'\in[n'']_y$, we can uniquely recover the original graph $\GG$ (modulo isomorphism) by gluing back the tree $\phi(\Tin)$, so we write $\GG=\GG(\cGG,Y')$ (in this graph, $\phi(\tL)$ is labelled). We then denote
	\beq\label{e:with.local.girth.cond}
	\bm{J}
	\equiv\bigg\{(\cGG,Y') \in \bm{C}' \times [n'']_y
	: \textup{in $\GG=\GG(\cGG,Y')$,
		the neighborhood $B_{4R}(a;\GG)$ is a tree}
	\bigg\}\,.
	\eeq
Recall from above that under $\P_{n,m}(\cdot\,|\,\GG\in\bm{C}_a)$, the induced law of $\cGG$ is equivalent (modulo graph isomorphism) to $\P_{n',m'}(\cdot\,|\,\cGG\in\bm{C}')$. Recall also that $\bm{E}_a$ is the same as $\bm{C}_a$ with the added restriction that $B_{4R}(a;\GG)$ is a tree. It follows that if $\bm{K}$ is any event that is invariant under graph isomorphism, then we have
	\beq\label{e:sample.cavity.graph.up.to.isom}
	\P\Big( (\cGG,Y')\in\bm{K} \,\Big|\, \GG\in \bm{E}_a
	\Big)
	=\P_{n',m'}\Big( (\cGG,Y')\in\bm{K} \,\Big|\,\bm{J} \Big)\,.\eeq
Conditional on $\bm{C}'$, it is clear that $Y'$ is a uniformly random element of $[n'']_y$, which is independent of the structure of $\cGG$, and hence independent of $S(\cGG)$: for any $Y_0\in[n'']_y$ and any subset $S_0\subseteq[n']$,
	\[
	\P_{n',m'}\Big( Y'=Y_0 \,\Big|\,\GG'\in\bm{C}', S(\cGG)=S_0\Big)
	= \f1{(n'')_y}\,.
	\]
We now argue that this does not change much if we condition further on $B_{4R}(a;\GG)$ being a tree. Indeed, since all the random graphs we consider are locally tree-like, it is clear that the event $\bm{J}$ of \eqref{e:with.local.girth.cond} occupies $1-o_n(1)$ fraction of $\bm{C}' \times [n'']_y$. This implies that, for any $Y_0\in[n'']_y$ and any $S_0\subseteq[n']$, we have
	\begin{align}\nonumber
	p(Y_0|S_0)
	&\equiv \P_{n',m'}\Big( Y'=Y_0 \,\Big|\, \bm{J}, S(\cGG)=S_0\Big)
	=\f{\P_{n',m'}( Y'=Y_0,\bm{J} \,|\, \GG'\in\bm{C}',
		 S(\cGG)=S_0)}
		 {\P_{n',m'}( \bm{J} \,|\, \GG'\in\bm{C}',
		 S(\cGG)=S_0)}\\
	&\le \f{\P_{n',m'}( Y'=Y_0\,|\, \GG'\in\bm{C}',
		 S(\cGG)=S_0)}
		 {1-o_n(1)}
	= \f{1-o_n(1)}{(n'')_y}\,.\label{e:exchangeability.Y}
	\end{align}
By a very similar argument, we can extend \eqref{e:small.removal.without.girth.cond} to
	\beq\label{e:small.removal.with.girth.cond}
	\P_{n',m'}\bigg(|S(\cGG)| \ge \f{n'}{\exp(2^{ck}R)}
	\,\bigg|\, \bm{J} \bigg) = o_n(1)\,.\eeq
Recall that $|Y| \le \exp(O(k^2 R))$. It follows by combining \eqref{e:small.removal.with.girth.cond}~and~\eqref{e:exchangeability.Y} that 
	\begin{align*}
	\P_{n',m'}\Big(Y' \cap S(\GG') \ne\emptyset
	\,\Big|\, \bm{J} \Big)
	&\le o_n(1)
		+ \P_{n',m'}\bigg(Y' \cap S(\GG') \ne\emptyset
	\,\bigg|\, \bm{J},|S(\GG')| < \f{n'}{\exp(2^{ck}R)} \bigg)
	\\
	&\le o_n(1)+ \f{\exp(O(k^2 R))}{\exp(2^{ck}R)} \le \f18\,.
	\end{align*}
The lemma follows by the equivalence \eqref{e:sample.cavity.graph.up.to.isom}.
\end{proof}
\end{lem}

\begin{proof}[Proof of Proposition~\ref{p:posfrac}] 
Let $\GG=(V,F,E)\sim\P=\P_{n,m}$. Let $\bL$ be any feasible type, in the sense of Definition~\ref{d:feasible.clause.type}. For any fixed $\bL$ we consider the events $\bm{E}_a(\bL)$, $\bm{G}_a(\bL)$, $\bm{F}_a(\bL)$ as in Definition~\ref{d:pos.frac.events}. Let
	\[
	J(\bL)\equiv
	\sum_{a\in F} J_a(\bL)
	\equiv
	\sum_{a\in F}\mathbf{1}\Big\{ \bm{E}_a(\bL)
		\cap \bm{G}_a(\bL) \cap \bm{K}_a(\bL)\Big\}\,.
	\]
It follows by combining Lemmas~\ref{l:lbd.Ea}--\ref{l:prob.Ka.given.Ea} that for all feasible $\bL$ we have
	\[
	\E \bm{J}(\bL)
	\ge \f{3nc_0(k,R)}{4}\,.
	\]
We will argue that each $\bm{J}(\bL)$ is sufficiently concentrated around its mean, such that 
	\beq\label{e:pos.frac.sufficient}
	\P\bigg( J(\bL) \ge \f{n c_0(k,R)}{2}\bigg) = o_n(1)\,.
	\eeq
Note that \eqref{e:pos.frac.sufficient} implies the result: as discussed in the proof of Lemma~\ref{l:lbd.Ea}, the total number of feasible types $\bL$ is upper bounded by some $C_0(k,R)$, so \eqref{e:pos.frac.sufficient}
(together with the fact that $\GG$ has girth greater than $8R$ with probability at least $c_0(k,R)$) implies
	\[\P\bigg( J(\bL) \ge \f{n c_0(k,R)}{2}
		\textup{ for all feasible }\bL
		\,\bigg|\, \girth(\GG)>8R
		\bigg) = o_n(1)\,.
	\]
In light of Corollary~\ref{c:event.implies.desired.type}, this directly implies the result. It therefore remains only to prove the concentration result \eqref{e:pos.frac.sufficient}. Let $\tL$ be the tree based on $\bL$ that is given by Definition~\ref{d:tree.based.on.feasible.type}, and recall from Definition~\ref{d:pos.frac.events} that $\bm{E}_a(\bL)$ is the event that there is a certain embedding $\phi_a:\tL\hookrightarrow\GG$ mapping $a_\star\mapsto a$, together with some local girth condition. The events $\bm{E}_a(\bL)$ and $\bm{G}_a(\bL)$ are clearly local --- they do not depend on more than $B_{4R}(a;\GG)$. By contrast, the event $\bm{K}_a(\bL)$ concerns preprocessing on the graph
	\[\cGG(a)
	\equiv\GG\Big\backslash \phi_a\Big( \Tin\setminus Y\Big)\,,\]
which (\textit{a~priori}) cannot be locally determined. This can be addressed with a similar argument as for Proposition~\ref{p:small.fraction.removed.in.processing}: as we saw in that proof, preprocessing is in fact fairly localized with high probability. In particular, we can define
	\[
	\GG^\circ(a)
	\equiv \Big\{ B_{(\log n)/2^{k\DELTACONST/6}}
		(a;\GG)\Big\}
	\Big\backslash \phi_a\Big( \Tin\setminus Y\Big)
	\subseteq
	\cGG(a)\,.
	\]
Let $A^\circ(a)$ be the variables $u\in\GG^\circ(a)$ with
$B_R(u;\cGG(a))\subseteq\GG^\circ(a)$, that fail to proper or $1$-good with respect to $\GG^\circ(a)$, and let
(cf. \eqref{e:equiv.defn.Ka})
	\[
	\overline{\bm{K}}_a(\bL)
	=\bm{E}_a(\bL)\cap \bigg\{\phi_a(Y)
	\textup{ does not intersect }
	B_R\bigg( 
	\bsp'\Big( A^\circ(a) ;\GG^\circ(a)\Big)
	; \GG^\circ(a)\bigg)
	\bigg\}
	\]
Let $\bar{J}(\bL)$ be defined as $J(\bL)$, but with $\overline{\bm{K}}_a(\bL)$ in place of $\bm{K}_a(\bL)$. It follows from the proof of Proposition~\ref{p:small.fraction.removed.in.processing} that $J(\bL)$ and $\bar{J}(\bL)$ agree with high probability, and so
	\[\E\bar{J}(\bL)
	=\E J (\bL) - o(n)
	\ge \f{3n c_0(k,R)}{4} - o(n)\,.
	\]
For any pair of distinct clauses $a\ne b$, their $(\log n)/2^{k\DELTACONST/6}$-neighborhoods do not intersect with high probability, and so are roughly independent. It follows that $\bar{J}(\bL)$ has variance $o(n^2)$, and so by Chebychev's inequality
	\[
	\P\bigg(\bar{J}(\bL) < \f{n c_0(k,R)}{2}\bigg)
	\le \P\bigg( \Big |\bar{J}(\bL)-\E\bar{J}(\bL)\Big| \ge 
		 \f{n c_0(k,R)}{5}\bigg)
	\le o_n(1)\,.
	\]
Since $\P(J(\bL)\ne\bar{J}(\bL))=o_n(1)$, this proves \eqref{e:pos.frac.sufficient}, and the result follows as argued above.
\end{proof}

\subsection{Uniformity of processed graph}\label{ss:unif}

We conclude this section with the proof of Proposition~\ref{p:unif}.

\begin{ppn}\label{p:switch.process.commutation}
Fix any $\ksat$ instance $\GG=(V,F,E)$ that can arise under the measure $\P=\P^{n,\alpha}$. In the processed graph $\proc\GG$, choose any two edges $e_i=(a_iu_i)$ ($i=1,2$) having the same total type (Definition~\ref{d:cpd.type}). Let $\switch$ denote the switching operation in which we cut the edges $e_i$, and reconnect the resulting half-edges as $f_1=(a_1u_2)$ and $f_2=(a_2u_1)$. Then the switching and processing operations commute:
	\[\proc\Big(\switch\GG\Big)
	=\switch\Big(\proc\GG\Big)\,.\]
In particular, this implies that $f_i$ survive in $\proc(\switch\GG)$.
\end{ppn}

We begin with a preliminary lemma:

\begin{lem}\label{l:init.set.agree.after.switch}
 In the setting of Proposition~\ref{p:switch.process.commutation},
the initial set (of variables that are improper or not $1$-good)
in $\GG$ is the same as in $\switch\GG$.

\begin{proof} Denote the initial set (of variables that are improper or not $1$-good) as $A_{-1}$ in $\GG$, and $S_{-1}$ in $\switch\GG$. Membership of a variable in $A_{-1}$ (resp.\ $S_{-1}$) is a property of its $R$-neighborhood relative to $\GG$ (resp.\ $\switch\GG$). If $B_R(v;\GG)$ is the same as $B_R(v;\switch\GG)$, then $v$ belongs either to both sets $A_{-1}$ and $S_{-1}$, or to neither. 

If $B_R(v;\GG)$ is not the same as $B_R(v;\switch\GG)$, then it must be that $B_R(v;\GG)$ contains at least one of the $e_i$, and so $B_R(v;\switch\GG)$ contains at least one of the $f_i$. In this situation, we have two observations:
\begin{enumerate}[(I)] 
\item We must have $v\notin A_{-1}$ by the assumption that the $e_i$ survive in $\proc\GG$ --- in particular, this implies that $B_R(v;\GG)$ must be proper (acyclic, with no repeated $\LABEL$ markings). It follows that $B_R(v;\GG)$ must contain exactly one of the $e_i$ while being disjoint from the other --- otherwise, it would contain two variables with the same marking $\LABEL$, making $v$ improper.
\item Similarly, $B_R(v;\switch\GG)$ must contain exactly one of the $f_i$ while being disjoint from the other --- otherwise, there is a path $\gamma_\star\subseteq\GG\cap(\switch\GG)$ of length at most $2R-1$
that joins the $f_i$ in $\switch\GG$. In the graph $\GG$, either $\gamma_\star$ forms a cycle with one of the $e_i$, or it joins $e_1$ to $e_2$. In both cases $\gamma_\star$ will contain an improper variable that will cause one of the $e_i$ to be removed during processing on $\GG$.
\end{enumerate}
Since we only switch two edges $e_i=(a_iu_i)$ of the same type, if $B_R(v;\switch\GG)$ is also acyclic then it must be isomorphic to $B_R(v;\GG)$, so in this case $v\notin S_{-1}$. The last possibility is that $B_R(v;\switch\GG)$ contains a cycle $C$. By observation (II), we can suppose without loss that $B_R(v;\switch\GG)$ contains $f_1$ but not $f_2$. If $f_1$ is disjoint from $C$, then $C$ will also appear in
$B_R(u_i;\GG)$ for one of the $u_i$. If $f_1$ lies on $C$, then $C\setminus f_1$ is a path in $\GG$ joining $e_1$ to $e_2$. In both cases $C$ will contain an improper variable that will cause one of the $e_i$ to be removed during processing on $\GG$. Altogether this proves that $A_{-1}=S_{-1}$, as claimed.
\end{proof}
\end{lem}

\begin{proof}[Proof of Proposition~\ref{p:switch.process.commutation}]
Denote the initial set (of variables that are improper or not $1$-good) as $A_{-1}$ in $\GG$, and $S_{-1}$ in $\switch\GG$. We showed in Lemma~\ref{l:init.set.agree.after.switch} that $A_{-1}=S_{-1}$. Note also that $B_R(A_{-1};\GG)$ cannot include either of the $e_i$, again by the assumption that the $e_i$ survive in $\proc\GG$. It follows that $B_R(A_{-1};\GG)=B_R(A_{-1};\switch\GG)$, and consequently
	\beq\label{e:init.step.commutes}
	\proc_0\Big(\switch\GG\Big)
	=\Big(\switch\GG\Big) \Big\backslash
		B_R\Big(A_{-1};\switch\GG\Big)
	=\switch \bigg( \GG \Big\backslash
		B_R\Big(A_{-1};\GG\Big)\bigg)
	= \switch\Big(\proc_0\GG\Big)\,,
	\eeq
that is to say, switching commutes with the initial preprocessing step.
For $t\ge0$ let us abbreviate
	{\setlength{\jot}{0pt}\begin{align*}
	A_t &\equiv \ACT(\proc_t\GG)\,,\\
	S_t &\equiv \ACT(\proc_t\switch\GG)\,.
	\end{align*}}%
If $A_\ell=S_\ell$ for all $0\le \ell<t$, then the same logic that led to \eqref{e:init.step.commutes} gives
	\beq\label{e:commute.up.to.t}
	\proc_t\Big(\switch\GG\Big)
	=\switch\Big(\proc_t\GG\Big)\,.\eeq
Thus, let $t$ be the first time that $A_t\ne S_t$; we will argue that this leads to a contradiction. Abbreviate $\HH\equiv\proc_t\GG$. \smallskip

\noindent\textbf{Case 1. $A_t\setminus S_t\ne\emptyset$.} 
Recall from the definition \eqref{e:bspprime.activated.set} that
membership of a variable in $A_t$ or $S_t$ is a property of its $(3R/10)$-neighborhood. Consequently, if there is any variable $v\in A_t\setminus S_t$, then $B_{3R/10}(v;\HH)$ must contain at least one of the $e_i$. But the very next step of processing on $\GG$ will remove $B_R(v;\HH)$, contradicting the assumption that both edges $e_i$ survive in $\proc\GG$. This proves that $A_t\subseteq S_t$.\smallskip

\noindent\textbf{Case 2. $S_t\setminus A_t\ne\emptyset$.}
Suppose $v\in S_t\setminus A_t$. From the above definitions and the relation \eqref{e:commute.up.to.t}, we have $A_t=\ACT(\HH)$ and $S_t=\ACT(\switch\HH)$. Thus, since $v\in S_t$, it must be that $B_{3R/10}(v;\switch\HH)$ contains at least one clause of degree $\le k-2$, or two clauses of degree $k-1$. On the other hand, since $v\notin A_t$, the same statement must not hold for $B_{3R/10}(v;\HH)$. It follows that $B_{3R/10}(v;\switch\HH)$ must contain at least one of the switched edges $f_i$. In fact, by observation~(II) in the proof of Lemma~\ref{l:init.set.agree.after.switch}, it must contain exactly one of the $f_i$, say $f_1$.

\begin{enumerate}[(a)]
\item If $B_{3R/10}(v;\switch\HH)$ contains any clause $b_\star$ of degree $\le k-2$, then the path joining $v$ to $b_\star$ must pass through at least one of the $f_i$, so the distance between $b_\star$ and $\set{f_1,f_2}$ in $\switch\HH$ is less than $3R/10$. It follows that the distance between $b_\star$ and $\set{e_1,e_2}$ in $\HH$ is also less than $3R/10$, which again yields a contradiction since it means that at least one of the $e_i$ will be deleted in the next preprocessing step.

\item It remains to consider the case that $B_{3R/10}(v;\switch\HH)$ contains two clauses $b_1\ne b_2$, each of degree $k-1$. If $\pi_i$ is the path joining $v$ to $b_i$ in $B_{3R/10}(v;\switch\HH)$, then at least one of the $\pi_i$, say $\pi_1$, must pass through a switched edge. It follows that the distance between $b_1$ and $\set{f_1,f_2}$ in $\switch\HH$ is less than $3R/10$; and so the distance between $b_1$ and $\set{e_1,e_2}$ in $\HH$ is also less than $3R/10$. We now consider the path $\gamma$ in $B_{3R/10}(v;\switch\HH)$ that joins $b_1$ to $b_2$, and distinguish two cases:

\begin{enumerate}[(i)]
\item If $\gamma$ does not pass through either $f_i$, then it is also a path in $\HH$. It has length at most $6R/10$ and joins $b_1$ to $b_2$, and we noted above that the distance between $b_1$ and $\set{e_1,e_2}$ in $\HH$ is less than $3R/10$. It follows that the next preprocessing step in $\GG$ will remove the $R$-neighborhood of some variable on $\gamma$, and thereby also remove one of the $e_i$. This gives a contradiction.

\item Now consider the case that $\gamma$ passes through $f_1$. This situation is shown in Figure~\ref{f:switch.two.edges.in.path}. We label $f_1=(xz)$ and $f_2=(wy)$ where $\set{x,w}$ is the same as either $\set{a_1,a_2}$ or $\set{u_1,u_2}$. Then $e_1=(xy)$ and $e_2=(wz)$.
Let $\rho_1$ denote the path between $b_1$ and $x$, and $\rho_2$ the path between $b_2$ and $z$: these paths are the same in $\HH$ as in $\switch\HH$, and must satisfy
	\[
	\length(\rho_1) + \length(\rho_2) + \f12
	= \length(\gamma) \le \f{6R}{10}
	\]
where we use $\length$ to denote path length. Recall from Remark~\ref{r:distance} that an edge $e=(av)$ has length $1/2$, so a path joining two neighboring variables has length one. On the left-hand side above, the $1/2$ term accounts for the length of the edge $f_1=(x z)$. In the graph $\HH$, the path $\rho_1$ joins $b_1$ to $x$ without passing through $y$. Since we assume that $e_1=(xy)$ survives in $\proc\GG$, any further processing on $\GG$ cannot remove $b_1$ without leaving behind a clause on $\rho_1$ of degree less than $k$ that lies even closer to $x$. It follows that the final graph $\proc\GG$ contains a path $\bar{\rho}_1\subseteq\rho_1$ joining $x$ to a clause $\bar{b}_1$ of degree less than $k$. Likewise, $\proc\GG$ must contain a path $\bar{\rho}_2\subseteq\rho_2$ joining $z$ to a clause $\bar{b}_2$ of degree less than $k$. Now note that $x$ and $z$ lie on opposite sides of the bipartite graph (one is a clause while the other is a variable). On the other hand, by the assumption that the $e_i$ have the same total type, there must be an isomorphism
	\[
	\phi : B_R\Big(x;\proc\GG\Big) \to B_R\Big(w;\proc\GG\Big),\quad
	\phi(x)=w,\quad\phi(y)=z\,.
	\]
It follows that $B_R(w;\proc\GG)$ must contain a path $\phi(\bar{\rho}_1)$ that joins $w$ to a clause $\phi(\bar{b}_1)$ of degree less than $k$. Then in the graph $\proc\GG$ we have
	\[
	d\Big(\bar{b}_2,\phi(\bar{\rho}_1);\proc\GG\Big)
	= \length(\phi(\bar{\rho}_1)) + \length(\bar{\rho_2}) + \f12
	\le \length(\rho_1) + \length(\rho_2) + \f12 \le \f{6R}{10}\,.
	\]
This contradicts the fact that, by definition, $\ACT(\proc\GG)$ must be empty.
\end{enumerate}
\end{enumerate}\smallskip

\noindent \textbf{Conclusion.}
The above argument shows that in fact $A_t=S_t$ for all $t$, so that the same logic leading to \eqref{e:init.step.commutes} and \eqref{e:commute.up.to.t} gives
the desired conclusion, that $\proc(\switch\GG)=\switch(\proc\GG)$.
\end{proof}

\begin{figure}[h!]
\includegraphics[page=3]{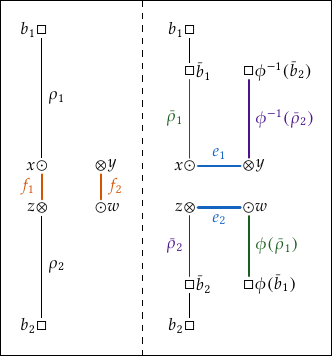}
\caption{Case (ii) in the proof of Proposition~\ref{p:switch.process.commutation}.
We let $t$ be the first time that $A_t\ne S_t$ and
$\HH=\proc_t\GG$, so that $\proc_t(\switch\GG)=\switch(\proc_t\GG)=\switch\HH$ by \eqref{e:commute.up.to.t}. In each of the panels \textsc{(a)}--\textsc{(c)}, the left side of the dashed line depicts part of $\switch\HH$ (including the switched edges $f_i$) while the right side depicts part of $\HH$ (including the original edges $e_i$). We use the symbols $\odot$ and $\otimes$ to indicate vertices from the two sides of the bipartite graph --- either $\odot$ indicates clause and $\otimes$ indicates variable, or the other way around (it does not matter which).}
\label{f:switch.two.edges.in.path}
\end{figure}

\begin{proof}[Proof of Proposition~\ref{p:unif}] 
Let $\GG\sim\P\equiv \P_{n,m}$. As in Proposition~\ref{p:switch.process.commutation}, let $e_1,e_2$ be two edges in the processed graph $\proc\GG$ having the same total type, and define the switching operation $\switch$. For any $H\in\ConfigModel(\cD)$,
	\begin{align*}
	\P(\proc\GG=H)
	&=\sum_G \P(\GG=G) \Ind{\proc G=H}
	\stackrel{\textup{(a)}}{=}
	\sum_G \P(\GG= G) \Ind{\proc(\switch G)=H}\\
	&\stackrel{\textup{(b)}}{=} 
	\sum_G \P(\GG= G) \Ind{\switch(\proc\GG)=H}
	= \P(\proc\GG=\switch H).
	\end{align*}
In the above, the step marked (a) follows from the fact the law of the original graph $\GG$ is invariant under the switching operation: $\P \circ \switch^{-1} = \P$. Equality (b) holds by Proposition~\ref{p:switch.process.commutation}, and equality (c) holds since the switching operation is involutive. Returning to the definition of $\ConfigModel(\cD)$ (Remark~\ref{r:processed.graph.types.CM}), the above proves $\P(\proc\GG=H)=\P(\proc\GG=H')$ for any $H,H'\in\ConfigModel(\cD)$. Finally, it is clear that the law of $\proc\GG$ is invariant under any permutation of the ordering among the variables or among the clauses, so if we condition on $\DD_{\proc\GG}=\DD$ then the law of $\cD_{\proc\GG}$ is uniformly random among all $\cD\sim\DD$.
This concludes the proof.
\end{proof}

\section{Extendibility and separability}\label{s:sep}

\noindent
In this section we prove Proposition~\ref{p:ext}
and Proposition~\ref{p:sep}. The section is organized as follows:
\begin{enumerate}[--]
\item In \S\ref{ss:explicit.lagrange.multipliers} 
we return to the issue that if a clause $a$ receives incoming messages $\dqstar_e$ ($e\in\delta a$), the resulting marginals on its incident edges does not generally agree with the canonical measures $\starpi_e$ from Definition~\ref{d:canonical}. We previously saw in \S\ref{ss:coherence.weights}
(specifically, Corollary~\ref{c:cohere.weights}) that if the clause is \bemph{coherent}, then it can be reweighted to achieve marginals $\starpi_e$. 
In \S\ref{ss:explicit.lagrange.multipliers} we show that if the clause is \bemph{nice}, then we can explicitly construct and estimate these weights.
This result will be used in the remainder of the current section, as well as in later sections.

\item In \S\ref{ss:planted.measure} we introduce the \bemph{planted measure} $\wP_{\DD}$. This is a well-known idea, which in the context of our problem gives a way to view the quantitites $\E_{\DD}\ZZ$,
$\E_{\DD}\extZZ$, and
$\E_{\DD}\sepZZ$
in terms of a certain modified configuration model
(Remark~\ref{r:sample.planted.msr}) that is tractable to analyze.

\item In \S\ref{ss:extendibility}
we prove Proposition~\ref{p:ext}, which can be reinterpreted as saying that most colorings are extendible under the planted measure. The idea of the proof is to show that, under the planted measure,
the subgraph of dependent free variables is with high probability a disjoint union of trees and unicyclic components, which can be completed to produce a valid satisfying assignment. The analysis of the free subgraph relies on the containment property
\eqref{e:every.variable.is.enclosed} of compound enclosures.

\item In \S\ref{ss:separability}
we prove Proposition~\ref{p:sep}.
The idea of the proof is to note that
if a (processed) $\ksat$ instance $\GG=(V,F,E)$ admits two judicious colorings that disagree on subset of variables $V'$ with $1 \ll |V'| \ll |V|$,
then $\GG$ must admit a certain combinatorial structure (Lemma~\ref{l:output.s.with.connections}) which we then show is unlikely to occur. The extraction of the compound structure relies on the fact that compound enclosures are bounded by variables that are perfect, hence orderly (Definition~\ref{d:orderly}).

\end{enumerate}

\subsection{Explicit Lagrange multipliers for nice clauses}
\label{ss:explicit.lagrange.multipliers}

Recall from \S\ref{ss:coherence.weights} (in particular, see Corollary~\ref{c:cohere.weights}) that whenever a clause $a$ is strictly coherent (Definition~\ref{d:coherence}), there exists a set of weights such that the associated Gibbs measure on colorings of $a$ has edge marginals $\starpi$ (Definition~\ref{d:canonical}). In this subsection, under the assumption that the clause is sufficiently nice, we give a direct construction and error estimate for these weights. This result will be used in the proofs that appear later in the current section. Moreover, the analysis in this subsection is a simplified version of the analysis of Section~\ref{s:contract}.

Through this subsection we are concerned with the reweighting of the edges around a single clause. Recall from \eqref{e:cyan.purple} that we introduced the composite color $\cya\equiv\SPIN{cyan}\equiv\set{\SPIN{green},\SPIN{blue}}$, for the reason that the clause factor \eqref{e:color.model.variable.factor} does not distinguish between $\SPIN{green}$ and $\SPIN{blue}$. Therefore in this subsection we can work on the reduced alphabet $\set{\red,\yel,\cya}$. Given a variable-to-clause message $\dq$ which is a probability measure over $\set{\RYGB}$, we now abuse notation and write $\dq$ for the measure on $\set{\red,\yel,\cya}$ where $\cya$ takes the combined weight of $\set{\grn,\blu}$. On the other hand, given a clause-to-variable message $\hq$ which is a probability measure over $\set{\RYGB}$
such that $\hq(\blu)=\hq(\grn)$, we define a probability measure $\tq$ on $\set{\red,\yel,\cya}$
with weights
	\beq\label{e:cyan.c.to.v}
	\Big(\tq(\red),\tq(\yel),\tq(\cya)\Big)
	= \bigg(
	\f{\hq(\red)}{1-\hq(\blu)},
	\f{\hq(\yel)}{1-\hq(\blu)},
	\f{\hq(\blu)}{1-\hq(\blu)}
	\bigg)\,.\eeq
For the rest of the subsection we work with the measures $\dq,\tq$ over $\set{\red,\yel,\cya}$.

As we will soon see, the bounds that we can obtain for the \textsc{bp} recursion are different for $\SPIN{red}$ versus the other colors. For this reason, given a variable-to-clause message $\dq\equiv\dq_{va}$, it will be useful to define a reweighted version
	\beq\label{e:reweighted.messages.single}
	Q_{va}(\sigma) 
	= \f{\dq(\sigma)}
		{(2^{|\pd a|-1})^{\Ind{\sigma=\red}}}
	\Bigg/
	\Bigg\{
	\f{\dq(\red)}{2^{|\pd a|-1}}
	+ \dq(\cya) + \dq(\yel)
	\Bigg\}\,.\eeq
With this notation, we can now state and prove the main technical result of this subsection:

\begin{lem}
\label{l:clause.bp.weights}
Let $\delta\in(0,1]$ be a fixed constant.
Suppose the clause $a$ receives incoming
variable-to-clause messages $\dq_e$ ($e\in\delta a$)
 whose reweightings $Q_e$ (as defined by \eqref{e:reweighted.messages.single}) satisfy the bounds
	\beq\label{e:bds.on.red.yellow}
	\max
	\bigg\{ Q_e(y)-\f12,
	Q_e(\red)
	\bigg\} \le \f1{2^{k\delta}}
	\eeq
for all $e\in\delta a$. Suppose also we have outgoing messages $\tq_e$ such that, for all $e\in\delta a$, we have
	\beq\label{e:assumed.bp.error.eps}
	\bigg|\f{[\BP_e[\dq]](\sigma)}
		{\tq_e(\sigma)}-1\bigg|
	\le \begin{cases}
	\ep_e 
		&\textup{for }\sigma\in\set{\yel,\cya}\,,\\
	\dot{\ep}_e
		&\textup{for }\sigma=\red\,,
	\end{cases}
	\eeq
where all the errors $\ep_e,\dot{\ep_e}$ are at most $1/k^4$. Then there exist edge 
weights $\Gamma\equiv(\gamma_e:e\in\delta a)$
with $\gamma_e(\yel)=1$ such that
	\beq\label{e:resulting.gamma.error.bound}
	\max_{e\in\delta a}
	\Bigg\{\sum_{\sigma\in\set{\red,\cya}}
	|\gamma_e(\sigma)-1|\Bigg\}
	\le k\sum_{e\in\delta a}
		\Big(\ep_e+\dot{\ep}_e\Big)\,,
	\eeq
and such that the $\Gamma$-weighted \textup{\textsc{bp}} recursion at clause $a$ maps $\dq$ to $\tq$: that is, such that
$\tq_e=\BP_e[\dq;\Gamma]$ for all $e\in\delta a$.

\begin{proof}
We will iteratively define a sequence of weights $\Gamma^t \equiv (\gamma_{e,t}: e\in\delta a)$, started from $\Gamma^0\equiv1$ and converging in the limit $t\to\infty$ to the desired weights $\Gamma$. We will maintain for all $t$ that $\gamma_{e,t}(\yel)=1$. Denote the output of the $\Gamma^t$-weighted recursion by
	\beq\label{e:weighted.bp}
	\tq_{e,t}(\sigma)
	\equiv \Big(\BP_e[\dq,\Gamma^t]\Big)(\sigma)
	\cong \gamma_{e,t}(\sigma)
		\bigg\{
		\Big(\BP_e[\gamma^t\dq]\Big)(\sigma)
		\bigg\}
		\,.\eeq

\noindent
\bemph{Step 1. Definition of weights and errors between weights.}
In this step, we fix an edge $e$ and abbreviate
$\gamma\equiv\gamma_e$, 
$\tq\equiv\tq_e$, and so on.
Given $\gamma^t$ and $\tq^t$,
we define the next weight $\gamma^{t+1}$ by setting
	\[\gamma^{t+1}(\sigma)
	\equiv
	\gamma^t(\sigma)
	\cdot
	\f{\tq(\sigma)/\tq(\yel)}
	{\tq^t(\sigma)/\tq^t(\yel)}
	\]
for each $\sigma\in\set{\red,\yel,\cya}$. Note this choice ensures $\gamma^{t+1}(\yel)=1$. It remains to estimate the error between $\gamma^t$ and $\gamma^{t+1}$ on the other two colors $\set{\red,\cya}$. To this end, let us define the error quantities
	\beq\label{e:def.eps.t.sigma}
	\ep^t(\sigma)
	\equiv
	\Bigg|
	\f{\gamma^{t+1}(\sigma)}{\gamma^t(\sigma)}-1
	\Bigg|
	=\Bigg|
	\f{\tq(\sigma)/\tq(\yel)}
	{\tq^t(\sigma)/\tq^t\yel)}
	-1\Bigg|\,.
	\eeq
for $\sigma\in\set{\red,\cya}$. 
The error at the next iteration is then given 
(after a short algebraic manipulation) by
	\beq\label{e:single.iterate.error}
	\ep^{t+1}(\sigma)
	=\Bigg|
	\underbrace{\bigg\{
	\f{\gamma^t(\sigma)}{\gamma^{t+1}(\sigma)}
	\f{\tq(\sigma)/\tq(\yel)}
	{\tq^t(\sigma)/\tq^t(\yel)}
	\bigg\}}_{\textup{equals one}}
	\f{\tq^t(\sigma)/(\gamma^t(\sigma)\tq^t(\yel))}
		{\tq^{t+1}(\sigma)/
			(\gamma^{t+1}(\sigma)
			\tq^{t+1}(\yel))}
	-1\Bigg]\,.\eeq
For the rest of this proof we use the shorthand
$\ep^t\equiv \ep^t(\cya)$ and $\dot{\ep}^t\equiv \ep^t(\red)$.\smallskip

\noindent \bemph{Step 2. Errors for variable-to-clause quantities.} In this step we continue to consider a single edge $e\in\delta a$, which we suppress from the notation. As in \eqref{e:reweighted.messages.single},
let $Q^t$ be the reweighted version of $\gamma^t\dq$ defined by
	\beq\label{e:def.Q.t}
	Q^t(\sigma)
	\equiv
	\f{\gamma^t(\sigma) \dq(\sigma)}
		{(2^{|\pd a|-1})^{\Ind{\sigma=\red}}}
	\Bigg/\Bigg\{
	\f{\gamma^t(\red)\dq(\red)}{2^{|\pd a|-1}}
	+\dq(\yel)
	+ \gamma^t(\cya)\dq(\cya)
	\Bigg\}\,.
	\eeq
Recall the assumption \eqref{e:bds.on.red.yellow} that
$Q(\red) = Q_0(\red) \le 2^{-k\delta}$. It will follow from the inductive analysis below that all the weights $\gamma^t$ are of constant order, so we will also have $Q^t(\red)\le O(2^{-k\delta})$. It follows that
	\begin{align*}
	\f{\gamma^{t+1}(\red)\dq(\red)}
		{2^{|\pd a|-1}}
	+\dq(\yel)
	+ \gamma^{t+1}(\cya)\dq(\cya)
	&=
	\f{\gamma^t(\red)\dq(\red)}{2^{|\pd a|-1}}
	\Big( 1 + O(\dot{\ep}^t)\Big)
	+\dq(\yel)
	+ \gamma^t(\cya)\dq(\cya)
	\Big( 1 + O(\ep^t)\Big)\\
	&=
	\Bigg\{
	\f{\gamma^t(\red)\dq(\red)}
		{2^{|\pd a|-1}}
	+\dq(\yel)
	+ \gamma^t(\cya)\dq(\cya)\Bigg\}
	\Bigg\{
	1 + O\bigg(\ep^t+
		\f{\dot{\ep}^t}{2^{k\delta}}\bigg)
	\Bigg\}\,,
	\end{align*}
for $\dot{\ep}^t$ and $\ep^t$
as in \eqref{e:def.eps.t.sigma}.
Combining this with \eqref{e:def.Q.t} gives
	\beq\label{e:Q.error.t}
	\Bigg|
	\f{Q^{t+1}(\sigma)}{Q^t(\sigma)}-1\Bigg|
	\le O(1)
	\begin{cases}
	\ep^t+
		\dot{\ep}^t/2^{k\delta}
	&\textup{for }\sigma\in\set{\yel,\cya}\,,\\
	\ep^t+\dot{\ep}^t
	&\textup{for }\sigma=\red\,.
	\end{cases}
	\eeq

\noindent\bemph{Step 3. Errors output by clause recursion.} Recall from \eqref{e:weighted.bp}
that $\tq_{e,t}/\gamma_{e,t}\cong\BP_e[\gamma^t\dq]$,
and recall from \eqref{e:def.Q.t} that $Q^t$ is a reweighted version of $\gamma^t\dq$. We therefore let $z_{e,t}$ denote the normalizing constant such that
	\[
	\f{\tq_{e,t}(\red)}{\gamma_{e,t}(\red)}
	= \f1{z_{e,t}}
	\prod_{e'\in\delta a \setminus e}
	Q_{e',t}(\yel)\,.\]
It the follows from the first bound in \eqref{e:Q.error.t} that for $\sigma=\SPIN{red}$ we have 
	\[\f{z_{e,t+1}\tq_{e,t+1}(\red)}
		{\gamma_{e,t+1}(\red)}
	=\prod_{e'\in\delta a \setminus e}
	Q^{t+1}_{e'}(\yel)
	=\f{z_{e,t}\tq_{e,t+1}(\red)}
		{\gamma_{e,t}(\red)}
	\Bigg\{ 1 + O\Bigg(
	\sum_{e'\in\delta a}
	\bigg(\ep_{e',t}
		+\f{\dot{\ep}_{e',t}}{2^{k\delta}}\bigg)
		\Bigg)\Bigg\}\,.
	\]
Next, by \eqref{e:Q.error.t} together with the assumption \eqref{e:bds.on.red.yellow}, we have for $\sigma=\SPIN{cyan}$ that
	\begin{align*}
	\f{z_{e,t+1}\tq_{e,t+1}(\cya)}
		{\gamma_{e,t+1}(\cya)}
	&=
	\prod_{e'\in\delta a \setminus e}
	\Big( 1-Q^{t+1}_{e'}(\red)\Big)
	-	\prod_{e'\in\delta a \setminus e}
	Q^{t+1}_{e'}(\yel)\\
	&=\prod_{e'\in\delta a \setminus e}
	\Bigg( 1-Q_{e',t}(\red)
		\bigg\{1+O\Big(
		\ep_{e',t}+
		\dot{\ep}_{e',t}
		\Big)\bigg\}
	\Bigg)-\prod_{e'\in\delta a \setminus e}\Bigg(
	Q_{e',t}(\yel)
	\bigg\{
	1 + O\bigg( \ep_{e',t}+
		\f{\dot{\ep}_{e',t}}{2^{k\delta}}
		 \bigg)
	\bigg\}\Bigg)\\
	&=\f{z_{e,t}\tq_{e,t}(\cya)}
		{\gamma_{e,t}(\cya)}
	\Bigg\{
	1 + O\Bigg(\sum_{e'\in\delta a}\bigg(
		\f{\ep_{e',t}+\dot{\ep}_{e',t}}{2^{k\delta}}
		 \bigg)\Bigg)
	\Bigg\}\,.
	\end{align*}
Lastly, for $\sigma=\SPIN{yellow}$, 
it follows from \eqref{e:bds.on.red.yellow} and 
\eqref{e:Q.error.t} that
	\begin{align*}
	z_{e,t+1}\tq_{e,t+1}(\yel)
	&=\prod_{e'\in\delta a \setminus e}
	\Big( 1-Q^{t+1}_{e'}(\red)\Big)
	-	\prod_{e'\in\delta a \setminus e}
	Q^{t+1}_{e'}(\yel)
		+
	\sum_{e'\in\delta a\setminus e}\Big(
	2^{|\pd a|-1}
	Q^{t+1}_{e'}(\red)
	- Q^{t+1}_{e'}(\cya)\Big)
	\prod_{e''\in\delta a \setminus \set{e,e'}}
	Q^{t+1}_{e''}(\yel)\\
	&=z_{e,t}\tq_{e,t}(\yel)
	\Bigg\{
	1 + O\Bigg(k\sum_{e'\in\delta a}\bigg(
		\f{\ep_{e',t}+\dot{\ep}_{e',t}}{2^{k\delta}}
		 \bigg)\Bigg)
	\Bigg\}\,.
	\end{align*}
Now substitute these estimates into the quantity appearing on the right-hand side of \eqref{e:single.iterate.error}: it gives
	\begin{align*}
	\ep_{e,t+1}(\cya)
	&=\Bigg|
	\f{\tq_{e,t}(\cya)/(\gamma_{e,t}(\cya)\tq_{e,t}(\yel))}
		{\tq_{e,t+1}(\cya)/
			(\gamma_{e,t+1}(\cya)
			\tq_{e,t+1}(\yel))}-1\Bigg|
	= O\Bigg(k\sum_{e'\in\delta a}\bigg(
		\f{\ep_{e',t}+\dot{\ep}_{e',t}}{2^{k\delta}}
		 \bigg)\Bigg)\,,\\
	\ep_{e,t+1}(\red)
	&=
	\Bigg|\f{\tq_{e,t}(\red)/(\gamma_{e,t}(\red)\tq_{e,t}(\yel))}
		{\tq_{e,t+1}(\red)/
			(\gamma_{e,t+1}(\red)
			\tq_{e,t+1}(\yel))}-1\Bigg|
	= O\Bigg(k\sum_{e'\in\delta a}\bigg(
		\ep_{e',t}+
		\f{\dot{\ep}_{e',t}}{2^{k\delta}}
		 \bigg)\Bigg)\,.
	\end{align*}
Consequently, if we aggregate all the error terms at time $t$ as
	\[
	E(t)
	\equiv \sum_{e\in\delta a} \bigg(
		\ep_{e,t}(\cya)
		+\f{\ep_{e,t}(\red)}{2^{k\delta/2}}\bigg)\,,
	\]
then we shall obtain at the next step
	\[
	E(t+1)
	\le 
	O(k^2)
	\sum_{e'\in\delta a}\bigg(
		\f{\ep_{e',t}+\dot{\ep}_{e',t}}{2^{k\delta}}\bigg)
	+\f{O(k^2)}{2^{k\delta/2}}
	\sum_{e'\in\delta a}\bigg(
		\ep_{e',t}+
		\f{\dot{\ep}_{e',t}}{2^{k\delta}}\bigg)
	\le \f{O(k^2)}{2^{k\delta/2}} E(t)\,.
	\]
It follows that $E(t)$ decays exponentially in $t$, so the sequence $\Gamma^t$ converges to a limit $\Gamma$.
Summing over $t$ gives the claimed error bound on the weights $\gamma_e$.
\end{proof}
\end{lem}

An application of Lemma~\ref{l:clause.bp.weights} is a direct construction, with an error estimate, for the weights of Corollary~\ref{c:cohere.weights} and Corollary~\ref{c:clause.bp.weights} in the case that the edges are nice:

\begin{cor}\label{c:clause.bp.weights.explicit}
In the setting of Corollaries~\ref{c:cohere.weights}~and~\ref{c:clause.bp.weights}, suppose all the edges in $U$ are stable and nice. Then, for every clause $a$ of $U$, the canonical messages
$\dqstar_e$ and $\hqstar_e$ (for $e\in\delta a$)
satisfy the conditions of Lemma~\ref{l:clause.bp.weights},
with $\delta=1/11$ in \eqref{e:bds.on.red.yellow}
and $\ep_e,\dot{\ep}_e \le O(1/k^r)$ in \eqref{e:assumed.bp.error.eps}. Consequently,
Lemma~\ref{l:clause.bp.weights} guarantees, for each $a$ in $U$, the existence of edge weights
$\Gmstar_a=(\gamma_e:e\in\delta a)$ 
satisfying
(cf.\ \eqref{e:resulting.gamma.error.bound}) $|\gamma_e(\sigma)-1|\le 1/k^{r/2}$ for all $\sigma$, and $\hqstar_e=\BP_e(\dqstar,\Gmstar_a)$ for all $e\in\delta a$. As a consequence, the weighted Gibbs measure \eqref{e:gamma.weighted.gibbs} has all edge marginals and messages agreeing with the canonical ones $\starpi,\dqstar,\hqstar$.
Redistributing the weights as in 
\eqref{e:redistribute.weights}
(and as in 
Corollary~\ref{c:clause.bp.weights})
produces a system
$\Lmstar$ of variable weights such that the $\Lmstar$-weighted Gibbs measure again has edge marginals $\starpi$. (However, since the weights were shifted from clauses to variables, the \textsc{bp} messages will be given by $\dqbul,\hqbul$
from \eqref{e:redistributed.bp.messages}
rather than by $\dqstar,\hqstar$). 

\begin{proof}
We need only to check that the conditions of Lemma~\ref{l:clause.bp.weights} are satisfied.
Indeed, for each edge $e\in U$,
if we take $\dqstar_e$ and 
use \eqref{e:reweighted.messages.single}
to define 
its reweighted version
$\Qstar_e$,
then $\Qstar_e$ will satisfy the error bound
\eqref{e:bds.on.red.yellow} with $\delta=1/11$,
because the edges in $U$ are assumed to be nice
(Definition~\ref{d:nice}).
Moreover, we will have the error bounds
\eqref{e:assumed.bp.error.eps}
with $\ep_e,\dot{\ep}_e \le O(1/k^r)$
because the edges in $U$ are assumed to be stable
(Definition~\ref{d:stable}).
The claimed result then follows directly from
Lemma~\ref{l:clause.bp.weights}. 
\end{proof}
\end{cor}


\subsection{The planted measure for judicious colorings} 
\label{ss:planted.measure}

In preparation for the proofs of
Propositions~\ref{p:ext}~and~\ref{p:sep},
we introduce the planted measure in this subsection. Given a processed neighborhood sequence $\cD$ (as in Definition~\ref{d:degseq}), let $\wP_\cD$ denote the uniform measure over all pairs $(\GG,\usi)$ such that $\GG\equiv(V,F,E)$ is consistent with the sequence $\cD$, and $\usi$ is a valid judicious coloring of $\GG$. Then, recalling Definition~\ref{d:extendible}
and \eqref{e:def.ZZ}, we have
	\beq\label{e:planted.measure.relation}
	\wP_\cD\bigg(\bigg\{
	(\GG,\usi) :
	\textup{$\usi$ is extendible}
	\bigg\}\bigg)
	= \f{\E_\cD\extZZ}{\E_\cD\ZZ}
	= \f{\E_{\DD}\extZZ}{\E_{\DD}\ZZ}
	\,,
	\eeq
where $\DD$ is the unordered version of $\cD$ (see Definition~\ref{d:degseq}). Thus, to prove Proposition~\ref{p:ext}, it suffices to show that $\usi$ is extendible with high probability under $\wP_\cD$. Similarly, to prove Proposition~\ref{p:sep},
it suffices to show 
that $\usi$ is separable with high probability under $\wP_\cD$. Following standard convention, we call $\wP_\cD$ the \bemph{planted measure}.

\begin{rmk}[sampling from the planted measure]\label{r:sample.planted.msr}
Recall from Proposition~\ref{p:unif}
that $\P_\cD$ coincides with the uniform measure over
$\ConfigModel(\cD)$, and can be sampled as a configuration model, as discussed in Remark~\ref{r:processed.graph.types.CM}.
It is well known that one can also sample from $\wP_\cD$ by a configuration-model-type procedure, as we now describe.
As in Remark~\ref{r:processed.graph.types.CM},
we fix a set of variables $V$ and a set of clauses $F$, equipped with incident half-edges $\delta V$ and $\delta F$, all labelled with types according to $\cD$. Then a graph $\GG$ corresponds to a matching $\mathfrak{M}$ of $\delta V$ to $\delta F$ that respects the edge types. The number of all such matchings is given by \eqref{e:size.CM.D}. Let $\bm{\sigma}$ be any coloring of $\delta V\sqcup\delta F$ that gives a valid coloring of $\delta x$ for every $x\in V\sqcup F$. We say that a valid coloring $\bm{\sigma}$ is \bemph{judicious},
abbreviated $\bm{\sigma}\in\mathcal{J}$, 
if the empirical measure of $\bm{\sigma}:\delta V\to\set{\RYGB}$ agrees with the canonical marginal $\starpi$, and the empirical measure of $\bm{\sigma}:\delta F\to\set{\RYGB}$ agrees with $\omstar$
(up to rounding, cf.\ Definition~\ref{d:judicious}). We write $\mathfrak{M}\sim\bm{\sigma}$ if the matching $\mathfrak{M}$ also respects the edge colors specified by $\bm{\sigma}$. Then a pair $(\GG,\usi)$
is equivalent to a pair $(\mathfrak{M},\bm{\sigma})$
such that $\mathfrak{M}\sim\bm{\sigma}$. Thus we can regard $\wP_\cD$ as the measure over pairs
$(\mathfrak{M},\bm{\sigma})$ given by
	\[
	\wP_\cD(\mathfrak{M},\bm{\sigma})
	=\f{\Ind{
	\bm{\sigma}\in\mathcal{J},
	 \mathfrak{M}\sim\bm{\sigma} }}
		{\mathcal{Z}}\,,
	\]
where $\mathcal{Z}$ is the normalizing constant. 
The marginal probability of $\bm{\sigma}$
under $\wP_\cD$ is given by
	\[
	\wP_\cD(\bm{\sigma})
	=\f{\Ind{\bm{\sigma}\in\mathcal{J}}}
	{\mathcal{Z}}\cdot
	\bigg|\Big\{ \mathfrak{M}
	:\mathfrak{M}\sim\bm{\sigma}
		\Big\}\bigg|
	=\f{\Ind{\bm{\sigma}\in\mathcal{J}}}{\mathcal{Z}}
	\cdot
	\prod_{\bt,\sigma}
	(n_{\bt} \cdot \starpi_{\bt}(\sigma))!\,,
	\]
which we emphasize is constant over all judicious
$\bm{\sigma}$. This implies
	\[\wP_\cD(\bm{\sigma})
	= \f{\Ind{\bm{\sigma}\in\mathcal{J}}}
		{|\mathcal{J}|}\,,
	\]
i.e., the marginal law of $\bm{\sigma}$ under $\wP_\cD$ is simply uniform over $\mathcal{J}$. It further implies that for any $\bm{\sigma}\in\mathcal{J}$,
	\[
	\wP_\cD(\mathfrak{M}\,|\,\bm{\sigma})
	= \f{\Ind{\mathfrak{M}\sim\bm{\sigma}}}
		{\mathcal{Z}/|\mathcal{J}|}
	= \f{\Ind{\mathfrak{M}\sim\bm{\sigma}}}
	{ | \set{\mathfrak{M}
		: \mathfrak{M} \sim\bm{\sigma} } | }\,,
	\]
i.e., under $\wP_\cD$, the law of $\mathfrak{M}$
conditional on $\bm{\sigma}$ is uniform among the matchings compatible with $\bm{\sigma}$. In conclusion, to generate a sample from $\wP_\cD$, 
we can first sample a uniformly random coloring $\bm{\sigma}\in\mathcal{J}$, then sample a uniformly matching $\mathfrak{M}$ that satisfies $\mathfrak{M}\sim\bm{\sigma}$.
\end{rmk}

\begin{rmk}[empirical measure of colorings
	under the planted measure]
\label{r:planted.measure.empir.msr}
Given a coloring $\bm{\sigma}\in\mathcal{J}$,
we 
associate the variable empirical measure
$\dbh$ as in \eqref{e:variable.color.cond.on.type},
and the clause empirical measure
$\hbh$ as in \eqref{e:clause.color.cond.on.type}.
Abbreviate $\bh\equiv(\dbh,\hbh)\equiv\bh(\bm{\sigma})$. Then, for any $
\bh$ that is consistent with $\starpi$ and $\omstar$, we have
	\[
	\bigg|\Big\{\bm{\sigma}\in\mathcal{J}
		: \bh(\bm{\sigma}) \Big\}\bigg|
	=\Bigg\{ \prod_{\bT}\binom{n_{\bT}}{n_{\bT}
		\dbh_{\bT}}\Bigg\}
	\Bigg\{ \prod_{\bL}
		\binom{m_{\bL}}{m_{\bL} \hbh_{\bL}}
	\Bigg\}\,.
	\]
By comparing with \eqref{e:config.model.first.moment.given.omega} and \eqref{e:nu.opt}, we see that if $\bm{\sigma}$ is sampled uniformly at random from $\mathcal{J}$, then its empirical measure $\bh(\bm{\sigma})$ is very close to $\optnu[\omstar]$ with high probability: more precisely, as long as
$\DD$ is bounded away from zero in the sense of \eqref{e:pos.frac}
(as guaranteed by Proposition~\ref{p:posfrac} with high probability), then
	\[\wP_\cD
	\bigg(\Big\|\nu- \optnu[\omstar]
		\Big\|_\infty
		\ge \f{\log n}{n^{1/2}}\bigg)
	\le \exp\bigg\{ - C(k,R) (\log n)^2\bigg\}\,.\]
We recall that
$\optnu[\omstar]
=(\optdbh[\omstar],\opthbh[\omstar])$ 
where $\dbh=\optdbh[\omstar]$
is given by Lemma~\ref{l:nu.star.as.optimizer.for.starpi},
while $\hbh=\opthbh[\omstar]$
is given by a reweighted measure:
if $a$ denotes a clause of type $\bL$, then
	\[
	\hbh_{\bL}(\usi_{\delta a})
	=
	\hat{\varphi}_a(\usi_{\delta a})
	\prod_{e\in\delta a}
	\bigg\{
	\dqstar_e(\sigma_e)\gamma_e(\sigma_e)
	\bigg\}
	\]
where the weights $\gamma_e$
are given by Corollary~\ref{c:cohere.weights},
since all clauses in the processed graph must be coherent. If all the edges in the clause are nice,
then the weights $\gamma_e$ are estimated by
Lemma~\ref{c:clause.bp.weights.explicit}.
We will use this observation in the proofs that follow.
\end{rmk}

We conclude this subsection with a simple lemma regarding the matching of edges near defects under the planted measure. It will be used in the analysis of \S\ref{ss:separability} below. We will say simply ``defect'' to refer to a $0$-defect in the graph $\GG$.

\begin{rmk}\label{r:defective.clauses} Recall from Definition~\ref{d:j.defective} that if a variable is non-defective, then it has at most one defective variable in its depth-one neighborhood. The clauses $a\in F$ therefore can be divided into three categories:
\begin{enumerate}[(i)]
\item The clause $a$ neighbors only defective variables, in which case we say that $a$ is a \bemph{defective clause}.
\item The clause $a$ has only one defective variable among its neighbors.
\item The clause $a$ has no defective variables among its neighbors, in which case we say that $a$ is a \bemph{strongly non-defective clause}.
(The strongly non-defective property will be used in the proof of Proposition~\ref{p:noNonDiverse}.)
\end{enumerate}
In both cases (ii) and (iii) we say that $a$ is a \bemph{non-defective clause}. If $v$ is a non-defective variable, then $\pd v$ contains only non-defective clauses; moreover, at most one clause in $\pd v$ can fail to be strongly non-defective. 
We say that $v$ is \bemph{strongly non-defective} if every clause in $\pd v$ is strongly non-defective.
We say that an edge $e$ is \bemph{strongly non-defective} if and only if $e\in\delta v$ for a strongly non-defective variable $v$. Lastly, we say that an edge $e$ is \bemph{internal} to a defect if it is incident to some internal clause of the defect.
\end{rmk}

\begin{lem}\label{l:matching.of.edges.near.defects}
As in Remark~\ref{r:sample.planted.msr}, 
fix $V,F,\delta V, \delta F$ labelled with total types according to $\cD$. From the notion of compound type (Definition~\ref{d:cpd.type}),
for each half-edge $e\in \delta V\sqcup \delta F$, we can deduce from its type $\bt_e$ whether, in the final graph $\GG$, the half-edge $e$ will participate in an edge that is internal to a defect. For any matching
$\mathfrak{m}$ of these half-edges, we have $\wP_\cD(\mathfrak{m}\subseteq\GG)=\P_\cD(\mathfrak{m}\subseteq\GG)$.

\begin{proof} As in the statement of the lemma, fix $V,F,\delta V, \delta F$ labelled with total types according to $\cD$. It is clear from Definition~\ref{d:cpd.type}
that for a half-edge $e$, the total type $\bt_e$ encodes whether $e$ will participate in an edge that is internal to a defect. That is to say, there is a set of edge types $\mathfrak{T}$ such that $\bt_e\in\mathfrak{T}$ if and only if $e$ participates in the internal edge of a defect. 

Let $\GG,\GG'$ be any two graphs with neighborhood sequence $\cD$. Let $\mathscr{H},\mathscr{H}'$ be the corresponding subgraphs induced by their defects; these are encoded by matchings $\mathfrak{m},\mathfrak{m}'$ as described in the statement of the lemma. The matchings involve precisely the half-edges with total types in $\mathfrak{T}$. 
Since defects are contained in enclosures which are encoded by compound types, it follows that $\mathscr{H}$ and $\mathscr{H}'$ are isomorphic
--- equivalently, that there is an isomorphism $\iota$ which takes $\mathfrak{m}\mapsto\mathfrak{m}'$.
Let $\GG$ be any graph in which $\mathfrak{m}$ appears. We can extend $\iota$ to $\GG$ by applying the identity map to edges not in $\mathfrak{m}$; then
 $\mathfrak{m}'$ appears in $\iota(\GG)$. If $\usi$ is a valid coloring of $\GG$, then a valid coloring of $\iota(\GG)$ is given by
 $(\usi\circ \iota^{-1})_e
\equiv \sigma_{\iota^{-1}(e)}$. If $\E_\cD$ denotes expectation under $\P_\cD$, then
	\begin{align*}
	\E_\cD\Big(\ZZ\,\Big|\,
		\mathfrak{m}\subseteq\GG\Big)
	&=\f1{\P_\cD(\mathfrak{m}\subseteq\GG)}
	\sum_{G : \mathfrak{m}\subseteq G}
		\P_\cD(G)
		\sum_{\usi}
		\Ind{\textup{$\usi$
			is a judicious coloring of $G$}}\\
	&=\f1{\P_\cD(\mathfrak{m}'\subseteq\GG)}
	\sum_{G : \mathfrak{m}\subseteq G}
		\P_\cD(\iota(G))
		\sum_{\usi}
		\Ind{\textup{$\usi\circ\iota^{-1}$
			is a judicious coloring of $\iota(G)$}}
	\\
	&=\f1{\P_\cD(\mathfrak{m}'\subseteq\GG)}
	\sum_{G' : \mathfrak{m}'\subseteq G'}
	\P_\cD(G')
	\sum_{\usi'}
	\Ind{\textup{$\usi'$
			is a judicious coloring of $G'$}}
	= \E_\cD\Big(\ZZ\,\Big|\,\mathfrak{m}'
		\subseteq\GG\Big)\,.
	\end{align*}
It follows from the definition of $\wP_\cD$ that
	\[
	\wP_\cD(\mathfrak{m}\subseteq\GG)
	= \f1{\E_\cD\ZZ}
	\sum_{G}
	\Ind{\mathfrak{m}\subseteq G}
	\P_\cD(G)
	\ZZ(\GG)
	=\P_\cD(\mathfrak{m}\subseteq\GG)
	\f{\E_\cD(\ZZ\,|\,\mathfrak{m}\subseteq\GG)}
	{\E_\cD\ZZ}
	=\P_\cD(\mathfrak{m}\subseteq\GG)\,,
	\]
as claimed.
\end{proof}
\end{lem}

\subsection{Extendibility for judicious colorings} 
\label{ss:extendibility}

In this subsection we prove Proposition~\ref{p:ext},
which says that the first moment $\E_\cD\ZZ$ of judicious colorings is dominated by extendible colorings (Definition~\ref{d:extendible}), with high probability over $\cD$.

\begin{dfn}[free subgraph]
Given a valid coloring $\usi$ of $\GG=(V,F,E)$, we define a subgraph $\mathfrak{F}
\subseteq\GG$ as follows. Let $V_{\mathfrak{F}}$ 
be the variables in $V$ that are incident to only 
$\SPIN{green}$ edges under $\usi$.
Let $F_{\mathfrak{F}}$ be the clauses in $F$
that are incident to only 
$\SPIN{green}$ or $\SPIN{yellow}$ edges under $\usi$.
Note that, by the rules of the coloring model, each $a\in F_{\mathfrak{F}}$ must be incident to at least two variables $u,v\in V_{\mathfrak{F}}$. Let $E_{\mathfrak{F}}$ be the edges between 
$V_{\mathfrak{F}}$ and $F_{\mathfrak{F}}$.
We shall call
$\mathfrak{F}
\equiv(V_{\mathfrak{F}}, F_{\mathfrak{F}},E_{\mathfrak{F}})$
the \bemph{free subgraph of $\GG$ induced by $\usi$.}
\end{dfn}

\begin{lem}\label{l:free.subgraph.no.bicycle}
In the setting of Proposition~\ref{p:ext} we have
	\[
	\wP_\cD\bigg(
	\Big\{
	(\GG,\usi) :
	\textup{any connected component of
	$\mathfrak{F}$
	contains a bicycle}
	\Big\}
	\bigg)
	= o_n(1)\,,
	\]
with high probability over $\cD$.

\begin{proof}
We assume throughout the proof that $\cD$ is fixed, such that its unordered version $\DD$ is bounded away from zero in the sense of
\eqref{e:pos.frac}
(as guaranteed by
Proposition~\ref{p:posfrac} with high probability). Then, as in Remark~\ref{r:sample.planted.msr},
 we first sample a uniformly random coloring
$\bm{\sigma}\in\mathcal{J}$, followed by a uniformly random matching $\mathfrak{M}$
such that $\mathfrak{M}\sim\bm{\sigma}$.
The resulting pair $(\mathfrak{M},\bm{\sigma})$
is equivalent to a sample
$(\GG,\usi)\sim\wP_\cD$. Let $\bh\equiv(\dbh,\hbh)$ be the empirical measure of $\usi$. We can assume that
	\beq\label{e:appx.nu.by.nustar}
	\Big\|\nu- \optnu[\omstar]
		\Big\|_\infty
		\le \f{\log n}{n^{1/2}}\,,
	\eeq
since this event holds with high probability under $\wP_\cD$ by Remark~\ref{r:planted.measure.empir.msr}.
Moreover, given $\bm{\sigma}$, the random matching $\mathfrak{M}$ can be explored in breadth-first manner, as we analyze next.

Fix an initial variable $v\in V$ of type $\bT$, and suppose under $\bm{\sigma}$ that it is incident to only $\SPIN{green}$ edges. 
Let $\mathfrak{F}(v)$ be the connected component of $\mathfrak{F}$ containing $v$.
Under the randomness of the matching $\mathfrak{M}$, the expected number of variables $w\in\mathfrak{F}(v)$
at unit distance from $v$ is
	\beq\label{e:free.branching}
	f(\bT) \equiv 
	\sum_{e\in\delta v}
	\sum_{\bL}
	\pi(\bL \,|\, \bt_e)
	\underbrace{
	\sum_{\usi_{\delta a}
		\in \set{\grn,\yel}^{\delta a}}
	\hbh_{\bL}
		(\usi_{\delta a} \,|\, \sigma_e=\grn)
	\sum_{e'\in\delta a\setminus e}
		\Ind{\sigma_{e'}=\grn}
		}_{\textup{denote this $g(\bt,\bL)$}}\,.
	\eeq
In the above $\pi(\bL\,|\,\bt)$ is the chance that an edge $e\in\delta v$ of type $\bt$ is matched with an edge $e'\in\delta a$ such that $\bL_a=\bL$. 
Recall from Definition~\ref{d:vfe} that the type $\bt$ includes the position $j(\bt)$ that the edge takes in the clause, so
	\beq\label{e:conditional.law.of.clause.type}
	\pi(\bL\,|\,\bt)
	= \hat{\DD}(\bL)
	\Bigg/\Bigg(
	\sum_{\bL'}
	\Ind{(\bL')_{j(\bt)} = \bt}
	\hat{\DD}(\bL')
	\Bigg)\,.
	\eeq
Consider the quantity $g(\bt,\bL)$ defined by \eqref{e:free.branching}: it is given explicitly by
	\beq\label{e:nice.free.branching}
	g(\bt,\bL)
	= 
	\f{\displaystyle
	\sum_{e'\in\delta a \setminus e}
	\dqstar_{e'}(\grn)\gamma_{e'}(\grn)
	\prod_{e'' \in \delta a \setminus
		\set{e,e'}}
		\bigg(
		\sum_{\sigma\in\set{\yel,\grn}}
	\dqstar_{e''}(\sigma)\gamma_{e''}(\sigma)
	\bigg)}{\displaystyle
	\prod_{e'' \in \delta a \setminus e}
	\bigg(
	\sum_{\sigma\in\set{\yel,\grn,\blu}}
	\dqstar_{e''}(\sigma)\gamma_{e''}(\sigma)
	\bigg)
	-\prod_{e'' \in \delta a \setminus e}
	\bigg(
	\dqstar_{e''}(\yel)\gamma_{e''}(\yel)
	\bigg)
	}
	\Bigg[ 1 + O\bigg( \f{\log n}{n^{1/2}}\bigg)
	\Bigg]\,,
	\eeq
where the error estimate comes from \eqref{e:appx.nu.by.nustar}.
If $\bL$ is a nice clause type, then we can use 
Definition~\ref{d:nice} together with
Corollary~\ref{c:clause.bp.weights.explicit} to bound $g(\bt,\bL) \le O(k/4^k)$.

If the variable $v$ is of a non-defective type $\bT$,
(Definition~\ref{d:j.defective}), then
it can only neighbor nice clause types $\bL$, and must have degree $|\delta v| = O(k2^k)$.
Thus we see from
\eqref{e:free.branching}
that $f(\bT) \le O(k^2/2^k)$. If $\bT$ is a
defective type,
then we recall from \eqref{e:every.variable.is.enclosed}
that $v$ must lie in a compound enclosure $U$ of diameter at most $\rprime$, such that the containment radius $s=\rad(v)$ is upper bounded by the distance between $v$ and the boundary of $U$. Moreover, by Definition~\ref{d:cpd.type}, the total type $\bT$ encodes the isomorphism class of $U$, the position of $v$ within $U$, and the simple total type of every edge in $U$. It follows that $\bT$ encodes the isomorphism class of $B_s(v)$, and the simple total type of every clause in $B_s(v)$. 
As before, suppose that $v$ is incident under $\bm{\sigma}$ to only $\SPIN{green}$ edges. Then,
under the randomness of the matching $\mathfrak{M}$,
the expected number of variables
in $\mathfrak{F}(v) \cap B_s(v)$ is
	\[
	f(\bT)
	\equiv
	\sum_{\ell=1}^s
	\sum_{v_0 a_1 v_1 \cdots a_\ell v_\ell}
	\prod_{j=1}^\ell
	\underbrace{\Bigg(
	\sum_{\usi_{\delta a_j} }
	\hbh_{a_j}
	\Big(\usi_{\delta a_j} \,\Big|\,
		\sigma_{v_{j-1} a_j}=\grn \Big)
	\mathbf{1}\Big\{
	\usi_{\delta a_j} 
	\in\set{\yel,\grn}^{\delta a_j},
	\sigma_{a_j v_j}=\grn
	\Big\}
	\Bigg)}_{g(v_{j-1},a_j,v_j)}\,,
	\]
where the second summation is over all paths
$v=v_0 a_1 v_1 \cdots a_\ell v_\ell$ emanating from $v$. If $a_j$ is a nice clause, then we have 
$g(v_{j-1},a_j,v_j) \le O(1/4^k)$
by a similar calculation as for \eqref{e:nice.free.branching}.
In any case we always have
$g(v_{j-1},a_j,v_j)\le1$. 
Recall from Definition~\ref{d:contained}
that $\mathfrak{B}(v,w)$ counts the number of defective variables on the shortest path between $v$ and $w$ (including the endpoints). 
By Definition~\ref{d:j.defective},
any clause neighboring to a non-defective variable must be nice, 
so the number of nice clauses
on the shortest path between $v$ and $w$
must be at least $d(v,w) - \mathfrak{B}(v,w)$.
It follows that
	\[
	f(\bT)
	\le
	\sum_{w\in B_s(v)}
	\bigg( \f{O(1)}{4^k}\bigg)^{d(v,w)-\mathfrak{B}(v,w) }
	\le \f14\,,
	\]
where the last inequality holds 
by the definition 
\eqref{eq-def-rad} of the containment radius,
since we set $s=\rad(v)$.

Now consider exploration of $\mathfrak{F}(v)$ by the following modified breadth-first search procedure.
We maintain a queue (a first-in first-out list) of variables, starting from $Q_0=(v)$. Then, at each time step $\ell\ge1$, we remove the first element $w_\ell$ of $Q_{\ell-1}$ to produce $(Q_{\ell-1})'$. We then explore the neighborhood of $w_\ell$ to
depth $s_\ell$ where $s_\ell=1$ if $w_\ell$ is non-defective, and $s_\ell=\rad(w_\ell)$ otherwise. Let $\mathfrak{F}(v)_\ell$ be the variables of $\mathfrak{F}(v)$ that are \emph{newly} discovered in this exploration. We then take the ones at the boundary of $B_{s_\ell}(w_\ell)$ and append them to the queue:
	\[
	Q_\ell 
	=\bigg( (Q_{\ell-1})'
	,\mathfrak{F}(v)_\ell \cap 
	\partial_\circ B_{s_\ell}(w_\ell)\bigg)\,.
	\]
The exploration continues until the first time $\tau(v)$ that $Q_{\tau(v)}=\emptyset$. Let $\zeta_\ell \equiv |Q_\ell|-|Q_{\ell-1}|$.
 Under the randomness of the matching $\mathfrak{M}$, 
 as long as $\ell/n = o_n(1)$
 and $Q_{\ell-1}\ne\emptyset$, 
 it follows from the preceding bounds
 on $f(\bT)$ that
	\[
	\E\zeta_\ell
	\le -1 + \f14 \bigg( 1 + o_n(1)\bigg)\,.
	\]
Once some vertices have already been explored,
\eqref{e:conditional.law.of.clause.type} is no longer an exact expression for the conditional 
law of subsequent clause types in the exploration
--- however, since we assume that $\ell/n=o_n(1)$ 
and that $\DD$ is bounded away from zero,
there remains a linear number of unexplored vertices of each type, so 
\eqref{e:conditional.law.of.clause.type}
is correct up to a multiplicative factor
$1 + o_n(1)$.
Thus, as long as 
$\ell/n=o_n(1)$
 and $Q_{\ell-1}\ne\emptyset$, 
we have $\E\zeta_\ell\le -1/2$.
Since all variables in $\GG$ must be fair 
(Definition~\ref{d:perfect.fair}),
we have $0 \le 1+\zeta_\ell \le \exp(k^2 R)$
with probability one. It follows by the Azuma--Hoeffding inequality that for a large enough constant $C=C(k,R)$,
	\[
	\P\Big(\tau(v) \ge C\log n\Big)
	\le
	\P\bigg( |Q_{C\log n}|
		-\E |Q_{C\log n}| \ge \f{C\log n}{2} \bigg)
	\ll \f1n\,.
	\]
On the other hand, again using that
$\DD$ is bounded away from zero,
we see that the chance for the exploration to close more than two cycles is at most $(\log n)^{O(1)}/n^2 \ll 1/n$. Taking a union bound over all $v\in V$ gives
	\[
	\wP_\cD\bigg(
	\Big\{
	(\GG,\usi) :
	\textup{any connected component of
	$\mathfrak{F}$
	contains a bicycle}
	\Big\}
	\bigg)
	= o_n(1)\,,
	\]
as long as 
$\DD$ is bounded away from zero
in the sense of \eqref{e:pos.frac}. The result follows by appealing to Proposition~\ref{p:posfrac}.
\end{proof}
\end{lem}

\begin{proof}[Proof of Proposition~\ref{p:ext}]
Let $\GG'\sim\P\equiv\P_{n,m}$ for $|m-n\alpha|\le n^{1/2}\log n$. Let $\GG\equiv\proc\GG'$ be the processed graph given by Definition~\ref{d:proc}.
Let $\usi$ be any valid coloring of $\GG$, and let $\ux$ be its corresponding frozen configuration: as long as the free subgraph of $\GG$ induced by $\usi$ does not contain a bicycle,
$\ux$ can be extended to a satisfying assignment of $\GG$, that is to say, $\usi$ is extendible
(Definition~\ref{d:extendible}).
Recalling \eqref{e:planted.measure.relation}, it follows that
	\begin{align*}
	\f{\E_\cD\extZZ}{\E_\cD\ZZ}
	&=\wP_\cD\bigg(\bigg\{
	(\GG,\usi) :
	\textup{$\usi$ is extendible}
	\bigg\}\bigg)\\
	&\ge
	\wP_\cD\bigg(\bigg\{
	\hspace{-3pt}
	\begin{array}{c}(\GG,\usi) :
	\textup{$\usi$ is a judicious coloring of $\GG$,
		and no connected}\\
	\textup{component of its free subgraph
		contains a bicycle}
	\end{array}\hspace{-3pt}
	\bigg\}\bigg)
	=1-o_n(1)
	\end{align*}
with high probability, where the last inequality is by 
Lemma~\ref{l:free.subgraph.no.bicycle}. This concludes the proof.
\end{proof}

\subsection{Separability for judicious colorings}
\label{ss:separability}

In this subsection we prove Proposition~\ref{p:sep}, which says that the first moment $\E_\cD\ZZ$ of judicious colorings is dominated by separable colorings (Definition~\ref{d:separable}), with high probability over $\cD$. We again let $\GG=(V,F,E)$ be the processed graph (with neighborhood sequence $\cD$), so that $\GG=\proc\GG'$ where $\GG'$ is the original $\ksat$ instance. We first show by a direct second moment calculation that the original instance $\GG'$ is very unlikely to have pairs of satisfying assignments of ``intermediate'' overlap:

\begin{lem}\label{l:standard.second.mmt}
Let $\GG'\sim\P_{n',n'\alpha}$ for $\alpha$ in the regime \eqref{e:alpha.regime}.
Recall \eqref{e:pairs.sat.assignments.overlap} that $Z^2[z]$ counts the number of pairs
$(\ux^1,\ux^2)$ of satisfying assignments of $\GG'$ with overlap $z$. Then, for any positive absolute constant $c$,
it holds for all 
	\[
	z \in 
	\bigg[0, \f12\bigg(1 - \f{ck}{2^{k/2}}
		\bigg)\bigg]
	\cup
	\bigg[\f12\bigg(1 - \f{ck}{2^{k/2}}\bigg),
	1-\f{k^2}{2^k}\bigg]\,,
	\]
we have
$\E_{n',n'\alpha} Z^2[z] 
	\le \exp\{ -\Omega(n' k^2/2^k) \}$.

\begin{proof}
Let $\E$ denote expectation under $\E_{n',n'\alpha}$.
Recall from 
\eqref{e:basic.second.moment.z}
that
	\[
	\E Z^2[z]
	= \f{\exp\{
		n' \ff_{\textsc{sat},2}(z)
	\}}{(n')^{O(1)}}
	= \f1{(n')^{O(1)}}
		\exp\Bigg\{ n'\bigg[
		\log2 + \Ent(z)
			+ \alpha\log\bigg( 1 - \f{2-z^k}{2^k}
		 \bigg)
	\bigg]
	\Bigg\}\,.
	\]
We use the inequality $\log(1-x) \le -x$ to bound
	\[\ff_{\textsc{sat},2}(z)
	\le \varphi(z) \equiv \log2 + 
	\Ent(z) - \alpha \bigg(
		\f{2-z^k}{2^k}\bigg)\,.
	\]
It suffices to prove $\varphi(z) \le - \Omega(k^2/2^k)$ for all $z$ in the claimed interval. By the restriction
\eqref{e:alpha.regime} on $\alpha$, we have
	\[
	\varphi(z)
	\le \Ent(z) - \log 2 + z^k \log2 
		+ O\bigg(\f1{2^k}\bigg)\,.
	\]
Next recall that $\Ent''(z) \le -4$ for all $z\in[0,1]$, so expanding around $z=1/2$ gives
	\[
	\varphi(z)
	\le - 2\bigg( z-\f12\bigg)^2 + z^k\log2\,.
	\]
This readily implies
$\varphi(z) \le -\Omega(k^2/2^k)$
for
	\[z \in 
		\bigg[0, \f12\bigg(1-\f{k}{2^{k/2}}\bigg)
		\bigg]
		\cup\bigg[
		\f12\bigg(1+\f{k}{2^{k/2}}\bigg)
		,\f12\bigg(1+\f{\log k}{k}\bigg)
		\bigg]\,.
	\]
It also implies $\varphi(z) \le -\Omega((\log k)^2/k^2)$ for
	\[
	z\in
	\bigg[
	\f12\bigg(1+\f{\log k}{k}\bigg),
	1-\f{2\log k}{k}
	\bigg]\,,
	\]
as well as $\varphi(z) \le -\Omega(1/k^{1/2})$ for
	\[z\in
	\bigg[1-\f{2\log k}{k},
	, 1-\f1{k^{3/2}}\bigg]\,.
	\]
Finally, a straightforward Taylor expansion near $z=1$ gives $\varphi(z) \le -\Omega( k^2(\log k)/2^k)$ for
	\[
	\ep \in\bigg[ 1-\f1{k^{3/2}},
	1- \f{k^2}{2^k} \bigg]\,,
	\]
concluding the proof.\end{proof}
\end{lem}

Lemma~\ref{l:standard.second.mmt} controls pairs of satisfying assignments of the original instance $\GG'$. The next result transfers this to a bound on pairs of judicious colorings of the processed instance $\GG=\proc\GG'$.

\begin{cor}\label{c:intermediate.overlap}
Recall from \eqref{e:second.mmt.decompose.by.z}
that $\ZZ^2[z]$ counts the number of pairs 
$(\usi^1,\usi^2)$ of judicious colorings of $\GG$, such that their corresponding frozen configurations
$\ux(\usi^i)$ have overlap $z$.
Let
	\[I_1\equiv\bigg[0,
		\f12\bigg(1-\f{k}{2^{k/2}}\bigg)
		\bigg]\cup 
		\bigg[ \f12\bigg(
			1+\f{k}{2^{k/2}}\bigg),
		 1-\f{k^2}{2^k}
		\bigg]\,,
	\]
and let $\ZZ^2[I_1]$ be the sum of $\ZZ^2[z]$ over $z\in I_1$. Then $\E_{\DD}[\ZZ^2[I_1]]
	\le \exp\{ -\Omega(nk^2/2^k) \}$ 
	with high probability over $\DD$.

\begin{proof} Let $\GG'=(V',F',E')$ with $|V'|=n'$, and $\GG=(V,F,E)=\proc\GG'$ with $|V|=n$. It follows from Proposition~\ref{p:small.fraction.removed.in.processing} that $n = n'(1-o_R(1))$ with high probability. Moreover, with high probability over $\cD$, in any judicious coloring the fraction of variables set to $\SPIN{free}$ is at most $4/2^k$. Given any frozen configuration $\ux$ of $\GG$, we can extend it to $\acute{\ux}\in\set{\minus,\plus,\free}^{V'}$ by setting $\acute{x}_v=\free$ for all $v\in V'\setminus V$. The resulting $\acute{\ux}$
is an ``almost-\textsc{sat} assignment'' in the sense that any clause in $\GG$ that does not neighbor a 
$\SPIN{free}$ variable must be satisfied. 
As long as $n = n'(1-o_R(1))$,
the number of $\SPIN{free}$ variables under $\acute{\ux}$ must be (crudely) at most $n'(5/2^k)$.

Suppose $(\usi^1,\usi^2)$ is a pair of judicious colorings of $\GG$, such that their corresponding frozen configurations $(\ux^1,\ux^2)$ agree on $nz$ variables. Define the extended configurations
$\acute{\ux}^i\in\set{\minus,\plus,\free}^{V'}$, and note that
	\[
	V'(\free)
	\equiv\bigg\{v\in V'
	: (\acute{x}^1)_v=\free 
	\textup{ or }
	(\acute{x}^2)_v=\free \bigg\}
	\]
has size $n'\pi \equiv |V'(\free)| \le n'(10/2^k)$.
Moreover, by definition we have
$V'(\free)\supseteq V'\setminus V$. Let
	\[
	V_=
	\equiv 
	\bigg\{v\in V'\setminus V'(\free)
	: (\acute{x}^1)_v=(\acute{x}^2)_v\bigg\}\,,
	\]
and note that
$n'(1-\pi)y\equiv|V_=|
\in[ nz-n'\pi, nz]$. Thus, with high probability over $\cD$, we have
	\beq\label{e:Zsq.decomposed.as.almost.sat}
	\ZZ^2[z]
	\le \sum_{\pi \le 10/2^k}
	\sum_{y: |y -z | \le 15/2^k }
	\bm{A}^2[\pi,y]
	\eeq
where $\bm{A}^2[\pi,y]$
counts the number of pairs $(\vec{\ux}^1,\vec{\ux}^2)$
 of almost-\textsc{sat} assignments with $|V'(\free)|=n'\pi$
and $|V_=| = n'y$. We then calculate
	\begin{align*}
	\E_{n',n'\alpha} \bm{A}^2[\pi,y]
	&\le 
	\binom{n'}{n'\pi}
	\binom{n'(1-\pi)}{n'(1-\pi)y}
	5^{n'\pi}
	2^{n'(1-\pi)}
	\Bigg(
	1- (1-\pi)^k
	\bigg( \f{2-y^k}{2^k} \bigg)
	\Bigg)^{n'\alpha}\\
	&
	\le\exp\Bigg\{
	 n'\bigg[ \ff_{\textsc{sat},2}(y) 
	 	+ O\bigg(\f{k}{2^k}\bigg) 
	 	\bigg]
	 \Bigg\}
	\le \exp\Bigg\{ -\Omega
		\bigg( \f{n'k^2}{2^k}\bigg)\Bigg\}\,,
	\end{align*}
where the last inequality follows by Lemma~\ref{l:standard.second.mmt}
for all $\pi,y$
in the range specified by \eqref{e:Zsq.decomposed.as.almost.sat}, when $z\in I_1$. It follows that
$\E_{n',n'\alpha}\ZZ^2[z] \le \exp\{-\Omega(n'k^2/2^k)\}$, and then applying 
Markov's inequality gives $\E_{\DD}[\bm{Z}^2[I_1]] \le \exp\{ -\Omega(nk^2/2^k) \}$ with high probability over $\cD$, as claimed.
\end{proof}
\end{cor}

To prove Proposition~\ref{p:sep} it remains to address pairs $(\ux^1,\ux^2)$ of frozen configurations with overlap in $[1-k^2/2^k,1]$, since the rest of $[0,1]$ is covered by $I_1$ from Corollary~\eqref{c:intermediate.overlap}, or by $I_0$ from \eqref{e:middle.interval}. We can also ignore overlaps very close to one, since for any $\ux^1$, the total number of configurations $\ux^2$ within Hamming distance $(\log n)^4$ is much smaller than $\exp\{(\log n)^5\}$. We thus restrict our attention to pairs $(\ux^1,\ux^2)$ of frozen configurations on $\GG$ with overlap in 
	\beq\label{e:separable.near.one.interval}
	I_2\equiv
	\bigg[
	1-\f{k^2}{2^k}, 1-\f{(\log n)^4}{n}
	\bigg]\,.\eeq
We begin with a combinatorial lemma which says that if such a pair exists, then the graph $\GG$ must (deterministically) contain a particular structure. In the remainder of this section we show that such structures are very unlikely to exist, from which Proposition~\ref{p:sep} will follow.

\begin{lem}\label{l:output.s.with.connections}
Let $\GG=(V,F,E)=\proc\GG'$ be a processed $\ksat$ instance, with $|V|=n$. Suppose that $(\ux^1,\ux^2)$ is a pair of frozen configurations on $\GG$, each corresponding to a judicious coloring of $\GG$, with
$\set{v\in V: (x^1)_v=(x^2)_v}=nz$ for some $z\in I_2$, as defined by \eqref{e:separable.near.one.interval}.
Then there exists a subset $S\subseteq V$,
	\beq\label{e:sep.bounds.on.size.of.S} 
	\f{(\log n)^3}{n}
	\le \f{|S|}{|V|} \le \f{k^2}{2^k}\,,
	\eeq
and for each $v\in S$ a set
$C(v)$ of directed paths
$u\to a\to v$ inside $\GG$, such that the following hold:

\begin{enumerate}[a.]
\item Let $\EPSCONST\equiv(\DELTACONST)^3$ 
(cf.\ Definition~\ref{d:orderly}).
Let $D$ be the defective variables in $S$,
and $B\equiv S\setminus D$. Then
$|D| \le 2 \EPSCONST|S|$.

\item For each $v\in S$, every element of $C(v)$
is of the form $u\to a\to v$
where $a$ is forcing to $v$ under $\ux^1$,
and $u\in S\setminus\set{v}$.

\item For each $v\in D$, the set $C(v)$ contains exactly one element $u\to a\to v$.

\item For each $v\in B$,
every clause $a$ that forces to $v$ under $\ux^1$ appears in exactly one element of $C(v)$.
\end{enumerate}
Let $C_S$ denote the set of all paths appearing in the sets $C(v)$ for $v\in S$.

\begin{proof}
We divide the proof into a few steps:\smallskip

\noindent\bemph{Step 1. Use $(\ux^1,\ux^2)$ 
to define an ``internally forced'' subset $X\subseteq V$.}
Since both the $\ux^i$
correspond to judicious colorings $\usi^i$,
they must each contain the same number of free variables,
from which it follows that
	\[
	\bigg|
	\bigg\{v\in V :
	(x^1)_v\ne\free,
	(x^2)_v=\free\bigg\}\bigg|
	=\bigg|
	\bigg\{v\in V :
	(x^1)_v=\free,
	(x^2)_v\ne\free\bigg\}\bigg|\,.
	\]
Consequently, if $d_H(\ux^1,\ux^2)
= |\set{v\in V : (x^1)_v \ne (x^2)_v}|$, the set
	\[
	X\equiv \bigg\{v\in V:
		(x^1)_v\ne (x^2)_v,
		(x^1)_v\ne\free\bigg\}
	\]
must have cardinality $|X| \ge d_H(\ux^1,\ux^2)/2$. The assumed overlap $z\in I_2$ then implies
	\beq\label{e:sep.bounds.on.X}
	\f{(\log n)^4}{2n}
	\le \f{|X|}{|V|}
	\le \f{k^2}{2^k}\,.
	\eeq
Since the $\ux^i$ differ on $X$
and $\ux^1$ is not $\SPIN{free}$ on $X$,
it must be that $X$ is \bemph{internally forced} with respect to $\ux^1$ --- that is to say, each $v\in X$ is forced, but only by clauses involving at least one other variable from $X$. This implies that for any $v\in X$ we can find a path
	\beq\label{e:forcing.path} 
	v=u_0 \leftarrow a_1 \leftarrow u_1
	\leftarrow a_2 \leftarrow u_2
	\leftarrow \ldots\eeq
where $a_i$ is forcing to $u_{i-1}$, and 
$u_i \in X$ with $u_i \ne u_{i-1}$. Since the graph is finite, any such path must eventually close on itself to form a directed cycle within $X$. It follows that every maximal connected component of $X$ must contain at least one cycle.\smallskip

\noindent
\bemph{Step 2. Extract $S\subseteq X$ 
which has a small fraction of defects, and 
is internally forced.} If $X$ does not intersect 
any compound enclosure
(Definition~\ref{d:enclosure})
in the processed graph $\GG$,
then $X$ does not contain any defective variable, and we simply take $S=X$.

If $U = U^\circ \sqcup \pd_\circ U$ is a compound enclosure and
 $X \cap U\ne\emptyset$, then we must have $X \cap \pd_\circ U\ne\emptyset$ --- indeed, if it were not the case, then $X\cap U = X\cap U^\circ$ would be a collection of maximal connected components of $X$, and each of these components would be a tree (since $U$ itself must be a tree). This contradicts our earlier observation that every maximal connected component of $X$ must contain at least one cycle.

We shall define a subset $S_U\subseteq X\cap U$ as follows. As we just saw, each connected component of $X\cap U$ induces a tree $\tree\subseteq\GG$ that intersects $\pd_\circ U$. Since $X$ is internally forced, the tree $\tree$ must be covered by forcing paths \eqref{e:forcing.path}; this makes $\tree$ into a \textsc{dag} (directed acyclic graph). Consider the variables of the \textsc{dag} with out-degree zero (meaning that they do not participate in any clause that forces another variable): if these all lie in $\pd_\circ U$, then we simply put
 all the variables of $\tree_U$ into $S_U$. 
Otherwise, suppose the \textsc{dag} contains a variable $v\notin \pd_\circ U$ with out-degree zero:
then we can replace $X$ by $X\setminus \set{v}$, which is also internally forced. We continue pruning variables in this way until we arrive at $X'\subseteq X$ such that in the \textsc{dag} corresponding to $X'\cap U$, all variables of out-degree zero lie in $\pd_\circ U$. We then set $S_U = X' \cap U$: by construction, $S_U$ contains $X \cap \pd_\circ U$. Moreover, if we take $S$ to be the union of 
all the perfect variables in $X$ together with $S_U$
for all compound enclosures $U$, then $S$ is internally forced.

Each variable in $S_U \cap U^\circ$ has positive out-degree (by the above construction) and positive in-degree (since $S$ is internally forced). This means that if we view $S_U$ as an undirected subgraph of $\GG$, each of its connected components must be a tree all of whose leaves lie in $\pd_\circ U$. It follows that $S_U$ can be covered by a disjoint union of (undirected) paths $\gamma$ where each $\gamma$ has a variable in $\pd_\circ U$ at one (or both) of its endpoints. 
Each $\gamma$ has length at most $\rprime$, the maximum diameter of $U$. By the construction of compound enclosures (Definition~\ref{d:enclosure}), each variable in $v\in \pd_\circ U$ is perfect, hence orderly (Definition~\ref{d:orderly}), so each $\gamma$ has at most $\EPSCONST$ fraction defective variables.
This implies that $S_U$ has at most $\EPSCONST$ fraction of defective variables. Moreover, if $\gamma$ contains any defective variable, then it must contain at least $1/\EPSCONST$ variables. It follows that the fraction of defective variables in $\gamma\setminus\pd_\circ U$ is at most
	\[
	\f{\EPSCONST \cdot (\textup{number of variables in $\gamma$})}{
	\textup{number of variables in 
	$(\gamma\setminus \pd_\circ U)$}
	}
	\le \f{\EPSCONST}{1-2\EPSCONST}
	\le 2\EPSCONST\,.
	\]
It follows that $S_U \setminus \pd_\circ U$ has at most $2\EPSCONST$ fraction of defective variables. Since every defective variable in $S$ must be contained in some $S_U \setminus \pd_\circ U$,
and the sets $S_U \setminus \pd_\circ U$ are pairwise disjoint, we conclude that $S$ has 
at most $2\EPSCONST$ fraction of defective variables. 
\smallskip

\noindent\bemph{Step 3. Choose paths $C(v)$ for $v\in S$.} Recall that $S=B\sqcup D$ where $D$ is the subset of defective variables in $S$. For $v\in D$, choose a single path $u\to a\to v$ such that $a$ is forcing to $v$ and $u\in S$, and let $C(v)$ consist of this path alone. For $v\in B$, for every clause $a$ that is forcing to $v$, choose a path $u\to a\to v$ with $u\in S$, and add this path to $C(v)$. Thus, for $v\in B$, every clause forcing to $v$ is covered by exactly one path in $C(v)$. This finishes the construction. Finally, since $S\subseteq X$ and $X$ satisfies \eqref{e:sep.bounds.on.X}, it is clear that $S$ satisfies the upper bound 
in \eqref{e:sep.bounds.on.size.of.S}, so it remains only to verify the lower bound. To this end, partition $X=X'\sqcup X''$ where $X'$ is the subset of all perfect variables in $X$. We can define a mapping $f: X''\to X'$ where we map $u\in X''$
to $v\in X'$ such that $v$ is on the boundary of the compound enclosure containing $u$. Recall that the maximum size of a compound enclosure is bounded by a constant depending only on $k$ and $R$. Moreover, a perfect variable must be nice, so its degree is $O(k2^k)$. This implies that for any $v\in X'$, its preimage $f^{-1}(v)\subseteq X''$ has size at most $C(k,R)$. It follows that $|X''| \le C(k,R) |X|$. Since $S\supseteq X'$ by construction, we conclude using the lower bound in \eqref{e:sep.bounds.on.X} that (for large $n$) we have
	\[
	\f{|S|}{|V|}
	\ge\f{|X'|}{|V|}
	\ge \f{(\log n)^4}{C(k,R) \cdot 2n}
	\ge \f{(\log n)^3}{n}\,,
	\]
as claimed. This concludes the proof.
\end{proof}
\end{lem}

In the remainder of this subsection, we show that the structure described by Lemma~\ref{l:output.s.with.connections} is unlikely to occur under the planted measure.
We make the following definitions based on the lemma:

\begin{dfn}[permissible tuples]\label{d:sep.permissible}
As in Remarks~\ref{r:processed.graph.types.CM}~and~\ref{r:sample.planted.msr}, we fix a set of variables $V$ and a set of clauses $F$, equipped with incident half-edges $\delta V$ and $\delta F$, all labelled with types according to $\cD$. We do not, as yet, take any matching of $\delta V$ to $\delta F$.
We say that the tuple $(u,j_u,a,j_v,v)$ is \bemph{permissible} if $u$ and $v$ are distinct variables, $a$ is a clause, 
and $j_u\ne j_v$ are indices such that
$\bL_a(j_u)\in\bT_u$
and $\bL_a(j_v)\in\bT_v$. There is a unique half-edge $e_u\in\delta u$ of type $\bt_{e_u}=\bL_a(j_u)$; we call this the \bemph{initial half-edge} of the tuple.
Likewise there is a unique half-edge
$e_v\in\delta v$ of type $\bt_{e_v}=\bL_a(j_v)$;
we call this the \bemph{final half-edge} of the tuple.

Let $\bm{\Pi}(\cD)$ denote the collection of all pairs $(S,P_S)$ such that $S$ is a subset of $V$ satisfying \eqref{e:sep.bounds.on.size.of.S},
and $P_S$ is a collection of permissible tuples
$(u,j_u,a,j_v,v)$ such that $u,v\in S$; every variable in $S$ appears as the final (i.e., fifth) entry of at least one tuple in $P_S$; and no clause appears in more than one tuple in $P_S$. Moreover we require that 
for every $v\in S$, all half-edges in $P(v)$ have the same sign $\lit$, where $P(v)$ denotes the set of all half-edges $e_v\in\delta v$
that arise as the final half-edge of a tuple in $P_S$.


Note that, given $\cD$, we can partition $S=B\sqcup D$ where $D$ denotes the defective variables in $S$.
Let $\bm{\Pi}_*(\cD)$ denote the subset of elements $(S,P_S)\in\bm{\Pi}(\cD)$ such that $|D| \le 2\EPSCONST|S|$;
 and each $v\in D$ appears as the final element of exactly one tuple in $P_S$.
\end{dfn}

\begin{dfn}[events based on permissible tuples]
\label{d:sep.events}
Again fix $V,F,\delta V,\delta F$, labelled with types according to $\cD$.
Take $(S,P_S)\in\bm{\Pi}(\cD)$ and partition $S=B\sqcup D$ as above. Let $F_S$ denote all the clauses appearing in $P_S$, and note that given $\cD$ we can also partition $F_S= F_B\sqcup F_D$ where $F_D$ denotes all the clauses in $F_S$ that are internal to defects, and $F_B\equiv F_S\setminus F_B$.
As in Remark~\ref{r:sample.planted.msr},
a pair $(\mathfrak{M},\bm{\sigma})$ 
with $\mathfrak{M}\sim\bm{\sigma}$ is equivalent to a pair $(\GG,\usi)$ where $\GG$ is a graph with neighborhood sequence $\cD$, and $\usi$ is a judicious coloring on $\GG$. We define the following events on the space of all pairs $(\mathfrak{M},\bm{\sigma})$.
 First recall that each clause $a\in F_S$ appears in a unique tuple
$(u,j_u,a,j_v,v)\in P_S$. Let
	\[
	\bm{Y}_a
	=\bigg\{(\mathfrak{M},\bm{\sigma})
	: \bm{\sigma}_{aj_v}=\red
	\textup{ and }
	\bm{\sigma}_{a \ell }=\yel
	\textup{ for all }
	\ell \ne j_v \bigg\}\,,\quad
	\bm{Y}_B
	\equiv \bigcap_{a\in F_B} \bm{Y}_a\,.
	\]
Membership of 
$(\mathfrak{M},\bm{\sigma})$ in $\bm{Y}_a$
depends only on $\bm{\sigma}_{\delta a}$. Let
	\[
	\bm{K}_a
	= \bigg\{
	\begin{array}{c}
	(\mathfrak{M},\bm{\sigma}):
	\textup{$\mathfrak{M}$
	contains a path}\\
	\textup{passing through
	$(u,a,v)$}
	\end{array}
	\bigg\}\,,\quad
	\bm{K}_B
	\equiv \bigcap_{a\in F_B}\bm{K}_a\,,\quad
	\bm{K}_D
	\equiv \bigcap_{a\in F_D}\bm{K}_a\,.
	\]
Let $\bm{K}_S\equiv\bm{K}_B \cap \bm{K}_D$. Membership in $\bm{K}_S$ depends only on the matching $\mathfrak{M}$, or equivalently the graph $\GG$, so we sometimes abuse notation and write simply
$\GG\in\bm{K}_S$. Next recall that for each $v\in S$, all the half-edges $P(v)\subseteq\delta v$ have the same sign $\lit$. Let
	\[
	\bm{X}_v
	\equiv \left\{
	(\mathfrak{M},\bm{\sigma}):
	\left\{ \begin{array}{c}
	\textup{$\bm{\sigma}_{\delta v}$
	is \SPIN{red} on all half-edges in $P(v)$,}\\
	\textup{$\SPIN{blue}$ on all edges in $\delta v(\plus\lit)\setminus P(v)$,}\\
	\textup{$\SPIN{yellow}$
		on all edges in $\delta v(\minus\lit)$}
	\end{array}
	\right\}\right\}\,,\quad
	\bm{X}_B
	\equiv \bigcap_{v\in B} \bm{X}_v\,.
	\]
Membership of 
$(\mathfrak{M},\bm{\sigma})$ in $\bm{X}_v$
depends only on $\bm{\sigma}_{\delta v}$.
\end{dfn}

It follows immediately from Lemma~\ref{l:matching.of.edges.near.defects} that
$\wP_\cD(\bm{K}_D)=\P_\cD(\bm{K}_D)$.

\begin{lem}\label{l:sep.colors.given.defect}
Fix a pair $(S,P_S)\in\bm{\Pi}_*(\cD)$
as in Definition~\ref{d:sep.permissible}.
For the events of Definition~\ref{d:sep.events}, we have
	\[
	\wP_\cD\Big(\bm{X}_B,\bm{Y}_B \,\Big|\,\bm{K}_D
	\Big)
	=\wP_\cD(\bm{X}_B)
	\wP_\cD(\bm{Y}_B)
	\le
	\bigg(\f{O(1)}{2^k}\bigg)^{|B| + 2|F_B|}\,.
	\]

\begin{proof}
Let $V,F,\delta V,\delta F$ be fixed as in Remark~\ref{r:sample.planted.msr}. The
event $\bm{K}_D$ depends on the matching of half-edges internal to defects.
The event $\bm{X}_B$
concerns the coloring $\bm{\sigma}$
on variable-incident half-edges that will not be internal to defects.
The event $\bm{Y}_B$
concerns the coloring $\bm{\sigma}$
on clause-incident half-edges
that will not be internal to defects.
It follows that the three events are mutually independent, therefore
	\beq\label{e:sep.independence.across.types}
	\wP_\cD\Big(\bm{X}_B,\bm{Y}_B \,\Big|\,\bm{K}_D
	\Big)
	=\wP_\cD(\bm{X}_B)\wP_\cD(\bm{Y}_B)\,.
	\eeq
It follows from Remark~\ref{r:sample.planted.msr} that
$\wP_\cD(\bm{X}_B)$
is the probability that a uniformly random element $\bm{\sigma}$ from $
\mathcal{J}$ satisfies the conditions of $\bm{X}_B$.
Suppose we instead sample $\bm{\sigma}:\delta V \sqcup \delta F\to\set{\RYGB}$ according to the measure
	\[\check{\wP}_\cD(\bm{\sigma})
	=\Bigg\{\prod_{v\in V} 
	\dbh_v(\bm{\sigma}_{\delta v})
	\Bigg\}\Bigg\{\prod_{a\in F}
	\hbh_a(\bm{\sigma}_{\delta a})
	\Bigg\}\,,\]
where $(\dbh,\hbh)\equiv\bh\equiv\optnu[\omstar]$.
It follows from 
Lemma~\ref{l:nu.star.as.optimizer.for.starpi}
and Corollary~\ref{c:cohere.weights} that
	\[\check{\wP}_\cD(\bm{\sigma})
	=\Bigg\{
	\prod_{v\in V} \Bigg(
	\f{ \varphi_v(\bm{\sigma}_{\delta v}) }{\dbz_v}
	\prod_{e\in\delta v}
	\hqstar_e(\bm{\sigma}_e)
	\Bigg)
	\Bigg\}\Bigg\{
	\prod_{a\in F}
	\Bigg(
	\f{\hat{\varphi}_a(\bm{\sigma}_{\delta a}}
	{ \hbz_a }
	\prod_{f \in \delta a}\bigg(
	\gamma_f(\bm{\sigma}_f)
	\dqstar_f(\bm{\sigma}_f)\bigg)
	\Bigg)\Bigg\}\,.\]
Moreover, if $a$ is a clause of type $\bL$
and $e$ is the $j$-th edge in $\delta a$,
then $\gamma_e$ depends only on $(\bL,j)$.
Thus each $\bm{\sigma}$ in $\mathcal{J}$ receives precisely the same weight under the measure $\check{\wP}_\cD$.
Moreover, if we sample $\bm{\sigma}\sim\check{\wP}_\cD$, then the \emph{expected} empirical measure of $\bm{\sigma}$ is $\bh=\optnu$.
Since the total number of types is at most $C(k,R)$, it follows by the local central limit theorem that
$\check{\wP}_\cD(\bm{\sigma}\in\mathcal{J}) = 1/n^{O(1)}$, so
	\[
	\wP_\cD(\bm{X}_B)
	\le n^{O(1)}
	\check{\wP}_\cD(\bm{X}_B)
	= n^{O(1)}
	\prod_{v\in B}
	\dbh_v( \bm{X}_v)\,.
	\]
A variable $v\in B$ must be non-defective, hence nice
(Definition~\ref{d:nice}). It follows that
	\begin{align*}
	\dbh_v( \bm{X}_v)
	&=\f1{\dbz_v}
	\prod_{e\in \delta v \cap P(v)}
		\hqstar_e(\red)
	\prod_{e\in \delta v(\lit) \setminus P(v)}
		\hqstar_e(\blu)
	\prod_{e\in \delta v(\minus\lit)}
		\hqstar_e(\yel)\\
	&\le O(1)
	\prod_{e\in \delta v \cap P(v)}
		\f{\hqstar_e(\red)}{\hqstar_e(\blu)}
	\prod_{e\in \delta v(\lit)}
		\f{\hqstar_e(\blu)}{\hqstar_e(\red,\blu)}
	\le \f{O(1)}
	{(2^k)^{1+|P(v)|}}\,.
	\end{align*}
Substituting into the previous calculation gives
	\beq\label{e:sep.colors.on.vars}
	\wP_\cD(\bm{X}_B)
	\le 
	n^{O(1)}
	\bigg(\f{O(1)}{2^k}\bigg)^{|B|+|F_B|}\,.
	\eeq
Similarly, a clause $a\in F_B$ can neighbor at most one defective variable, so it follows from 
Definition~\ref{d:j.defective}
that all the variables in the clause must be nice.
It then follows using 
Corollary~\ref{c:clause.bp.weights.explicit} that
	\[
	\hbh_a(\bm{Y}_a)
	=\f1{\hbz_a}
	\gamma_j(\red) \dqstar_j(\red)
	\prod_{\ell\ne j}\bigg\{
	\gamma_\ell(\yel)
	\dqstar_\ell(\yel)\bigg\}
	\le
	O(1)
	\f{\dqstar_j(\red)}
		{\dqstar_j(\set{\yel,\grn,\blu})}
	\prod_{\ell\ne j}
	\f{\dqstar_\ell(\yel)}
		{\dqstar_\ell(\set{\yel,\grn,\blu})}
	\le \f{O(1)}{2^k}\,.
	\]
It follows by a similar argument as for $\bm{X}_B$ that
	\beq\label{e:sep.colors.on.clauses}
	\wP_\cD(\bm{Y}_B)
	\le n^{O(1)}
	\check{\wP}_\cD(\bm{Y}_B)
	\le n^{O(1)}
	\bigg(\f{O(1)}{2^k}\bigg)^{|B|+|F_B|}\,.
	\eeq
The result follows by combining
\eqref{e:sep.independence.across.types},
\eqref{e:sep.colors.on.vars}, and \eqref{e:sep.colors.on.clauses}.
\end{proof}
\end{lem}

\begin{lem}\label{l:sep.B.edges.match.bias}
Fix a pair $(S,P_S)\in\bm{\Pi}_*(\cD)$
as in Definition~\ref{d:sep.permissible}.
For the events of Definition~\ref{d:sep.events}, we have
	\[
	\f{\wP_\cD(\bm{K}_B \,|\, 
	\bm{X}_B,\bm{Y}_B,\bm{K}_D)}
		{\P_\cD(\bm{K}_B \,|\, \bm{K}_D)}
	=\f{\wP_\cD(\bm{K}_B \,|\, \bm{X}_B,\bm{Y}_B)}
		{\P_\cD (\bm{K}_B)}
	\le \Big(O(1) 2^k\Big)^{|F_B|}
	\,.
	\]

\begin{proof}
Recall from Definition~\ref{d:sep.permissible} that each tuple $(u,j_u,a,j_v,v)\in P_S$
distinguishes an initial half-edge $e_u\in\delta u$
and a final half-edge $e_v\in\delta v$.
For each edge type $\bt$,
let $a_{\bt}(\yel)$
count the number of tuples in $P_S$
where the initial half-edge has type $\bt$.
Let $a_{\bt}(\red)$
count the number of tuples in $P_S$
where the final half-edge has type $\bt$.
 Let $a_{\bt}\equiv a_{\bt}(\red)+a_{\bt}(\yel)$.
 Then note that
	\[
	\sum_{\bt}
	a_{\bt}(\yel)
	=\sum_{\bt}
	a_{\bt}(\red)
	= \f12\sum_{\bt}a_{\bt}
	= |F_B|\,.
	\]
Since we restricted to clauses in $F_B$, any type $\bt$ with $a_{\bt}>0$ must be nice. The events $\bm{K}_B$ and $\bm{K}_D$ involve edges of distinct types, so they are independent under $\P_\cD$. It follows that
	\[
	\P_\cD(\bm{K}_B\,|\,\bm{K}_D)
	=\P_\cD(\bm{K}_B)
	=\prod_{\bt}
	\Bigg\{
	\prod_{i=0}^{a_{\bt}-1}
	\f1{n_{\bt}-i}
	\Bigg\}
	=\prod_{\bt}
	\f1{(n_{\bt})_{a_{\bt}}}
	\]
(using the standard notation for the falling factorial). On the other hand, under the measure $\wP_\cD$, if we condition on the events 
$\bm{X}_B$ and $\bm{Y}_B$, then the event $\bm{K}_B$ is more likely to occur because we have conditioned the edges involved to have compatible colorings:
	\[\wP_\cD\Big(\bm{K}_B\,\Big|\,
		\bm{X}_B,\bm{Y}_B,\bm{K}_D\Big)
	=\wP_\cD\Big(\bm{K}_B\,\Big|\,
		\bm{X}_B,\bm{Y}_B\Big)
	= \prod_{\bt}
	\f1{
	(n_{\bt}\cdot\starpi_{\bt}(\yel))_{a_{\bt}(\yel)}
	(n_{\bt}\cdot\starpi_{\bt}(\red))_{a_{\bt}(\red)}
	}\,,
	\]
which is clearly larger than the quantity
$\P_\cD(\bm{K}_B\,|\,\bm{K}_D)$ calculated just above. For any integers $0<a<b$ we have using Stirling's formula that
	\[
	b^a \ge (b)_a
	= \f{b!}{(b-a)!}
	\asymp
	\f{b^a}{e^a}
	\bigg(\f{b}{b-a}\bigg)^{1/2+b-a}
	\ge 
	\f{b^a}{e^a}\,.
	\]
It follows by combining the above that
	\[
		\f{\wP_\cD(\bm{K}_B \,|\, 
	\bm{X}_B,\bm{Y}_B,\bm{K}_D)}
		{\P_\cD(\bm{K}_B \,|\, \bm{K}_D)}
	=\f{\wP_\cD(\bm{K}_B \,|\, \bm{X}_B,\bm{Y}_B)}
		{\P_\cD (\bm{K}_B)}
	\le
	\prod_{\bt}
	\f{e^{a_{\bt}}}{
	(\starpi_{\bt}(\yel))^{a_{\bt}(\yel)}
	(\starpi_{\bt}(\red))^{a_{\bt}(\red)}
	}
	\le \Big( O(1) 2^k\Big)^{|F_B|}\,,
	\]
as claimed.
\end{proof}
\end{lem}

\begin{cor}\label{c:sep.main.bound}
Fix a pair $(S,P_S)\in\bm{\Pi}_*(\cD)$
as in Definition~\ref{d:sep.permissible}.
For the events of Definition~\ref{d:sep.events}, we have
	\[\f{\wP_\cD(
	\bm{X}_B,\bm{Y}_B,\bm{K}_S)}
		{\P_\cD(\bm{K}_S)}
	\le
	\bigg( \f{O(1)}{2^k}\bigg)^{|B|+|F_B|}
	\le
	\bigg( \f{O(1)}{2^k}\bigg)^{(1-2\EPSCONST)
		(|S|+|P_S|)}
	\]
\begin{proof}
The first inequality follows directly by combining
Lemmas~\ref{l:matching.of.edges.near.defects},
\ref{l:sep.colors.given.defect},
and \ref{l:sep.B.edges.match.bias}. Next, recall 
the assumption from Definition~\ref{d:sep.permissible}
 that
$|B| \ge (1-2\EPSCONST)|S|$. It follows that
	\[
	\f{|F_B|}{|F_S|}
	\ge
	\f{\displaystyle\sum_{v\in B} |P(v)|}{|F_S|}
	=\f{\displaystyle\sum_{v\in B} |P(v)|}{
	\displaystyle\sum_{v\in B} |P(v)| + |D|}
	\ge \f{|B|}{|B|+|D|}
	= \f{|B|}{|S|} \ge 1-2\EPSCONST\,,
	\]
which implies the second inequality
since $|F_S|=|P_S|$.
\end{proof}
\end{cor}

\begin{proof}[Proof of Proposition~\ref{p:sep}]
Take $I_1$ as in Corollary~\ref{c:intermediate.overlap},
and $I_2$ as in
\eqref{e:separable.near.one.interval}. 
Then,
recalling Definition~\ref{d:separable}
and \eqref{e:def.ZZ},
	\[\sepZZ
	\le
	\ZZ^2[I_1]
	+ \underbrace{ \sum_{\usi}
	\mathbf{1}\Bigg\{
	\bigg|
	\bigg\{
	\textup{$\usi'$}
	:\f{|v\in V: x(\usi)_v=x(\usi')_v|}{|V|}
	\in I_2
	\bigg\}
	\bigg| \ge \f{\exp\{(\log n)^5\}}{2}
	\Bigg\}
	}_{\textup{denote this $(\sepZZ)'$}}
	\,.
	\]
We have a bound on
$\E_{\DD}\ZZ^2[I_1]=\E_\cD\ZZ^2[I_1]$
from Corollary~\ref{c:intermediate.overlap}.
It follows from the definition of $\wP_\cD$ that
	\[
	r'(\cD)
	\equiv\f{\E_\cD(\sepZZ)'}{\E_\cD\sepZZ}
	=\wP_\cD\Bigg(
	\Bigg\{
	(\GG,\usi):
	\bigg|
	\bigg\{
	\textup{$\usi'$}
	:\f{|v\in V: x(\usi)_v=x(\usi')_v|}{|V|}
	\in I_2
	\bigg\}
	\bigg| \ge \f{\exp\{(\log n)^5\}}{2}
	\Bigg\}\Bigg)\,.\]
Recall that the events of Definition~\ref{d:sep.events} all depend on the choice of $(S,P_S)$. To make this explicit, we now write $\bm{K}[S,P_S]\equiv \bm{K}_S$
and $\bm{T}[S,P_S] \equiv \bm{X}_B \cap \bm{Y}_B \cap \bm{K}_S$. It then follows from Lemma~\ref{l:output.s.with.connections} 
and Corollary~\ref{c:sep.main.bound} that
	\[r'(\cD)
	\le\sum_{(S,P_S)\in\bm{\Pi}_*(\cD)}
		\wP_\cD(\bm{T}[S,P_S] )
	\le\sum_{(S,P_S)\in\bm{\Pi}_*(\cD)}
		\f{\P_\cD(\bm{K}[S,P_S] )}
		{2^{k(1-3\EPSCONST) (|S|+|F_S|)} }
		\,.
	\]
We now wish to take expectation over the law $\P=\P_{n',m'}$ of the original graph $\GG'=(V',F',E')$, of which $\GG=\proc\GG'$ is the processed version. To this end, let $\bm{\Pi}'$ denote the set of all pairs $(S,P_S)$ where
$S\subseteq V$ satisfies \eqref{e:sep.bounds.on.size.of.S}, and
 $P_S$ is a collection of tuples
$(u,j_u,a,j_v,v)$ where $u\ne v$ in $S$,
$a\in F'$, every $v\in S$ appears as the last entry of at least one element of $P_S$, and no clause appears more than once in $P_S$.
Note that if $(S,P_S)\in\bm{\Pi}'$ and
$\P_\cD(\bm{K}[S,P_S] )$ is positive, then in fact
$(S,P_S)\in\bm{\Pi}(\cD)$. Thus, writing $\E$ for expectation under $\P$, we have (by the tower property of conditional expectation)
	\begin{align*}
	\E r'(\cD)
	&\le
	\sum_{(S,P_S)\in\bm{\Pi}'}
	\E\bigg(
	\Ind{(S,P_S)\in\bm{\Pi}(\cD)}
	\f{\P_\cD(\GG \in \bm{K}[S,P_S] )}
		{2^{k(1-3\EPSCONST) (|S|+|F_S|)} }
	\bigg)
	= \sum_{(S,P_S)\in\bm{\Pi}'}
	\E\bigg(
	\f{\P_\cD(\GG\in\bm{K}[S,P_S] )}
		{2^{k(1-3\EPSCONST) (|S|+|F_S|)} }
	\bigg)\\
	&= \sum_{(S,P_S)\in\bm{\Pi}'}
	\f{ \P( \GG\in \bm{K}[S,P_S])}
		{2^{k(1-3\EPSCONST) (|S|+|F_S|)} }
	\le \sum_{(S,P_S)\in\bm{\Pi}'}
	\f{ \P( \GG'\in \bm{K}[S,P_S])}
		{2^{k(1-3\EPSCONST) (|S|+|F_S|)} }
	\,,
	\end{align*}
where the last step uses that if $\bm{K}[S,P_S]$ occurs for the processed graph $\GG$,
then it also occurs for the original graph $\GG'$.
Taking into account the number of choices for $S$ and $P_S$, as well as the probability for the edges to be present under $\P$, we can bound the above as
	\[
	\E r'(\cD)
	\le
	o_n(1)
	+
	\sum_s
	\mathbf{1}\Bigg\{
	\f{(\log n')^2}{n'}
	\le 
	s \le \f{k^2}{2^k}
	\Bigg\}
	\binom{n'}{n's}
	\f1{2^{k(1-4\EPSCONST)n's}}
	\Bigg(
	\sum_{\ell\ge1}
	\bigg(
	\f{n' s \cdot k^2 \cdot m' / (n')^2}
			{2^{k(1-4\EPSCONST)}}
	\bigg)^\ell
	\Bigg)^{n's}\,.
	\]
In the above, the $o_n(1)$ error accounts for the probability that more than $o_R(1)$ fraction of variables are removed during preprocessing, which is controlled by Proposition~\ref{p:small.fraction.removed.in.processing}. On the complementary event, the number of variables $n$ in $\GG$ is very close to the original number of variables $n'$, so the restriction on $s=|S|/n'$ follows from \eqref{e:sep.bounds.on.size.of.S}.
The inner sum over $\ell\ge1$ is for the possible sizes of the sets $P(v)$. Simplifying the above gives
	\begin{align*}
	\E r'(\cD)
	&\le
	\sum_s
	\mathbf{1}\Bigg\{
	\f{(\log n')^2}{n'}
	\le 
	s \le \f{k^2}{2^k}
	\Bigg\}
	\binom{n'}{n's}
	\f1{2^{k(1-3\EPSCONST)n's}}
	\Bigg(
	\sum_{\ell\ge1}
	(2^{4\EPSCONST k} s)^\ell
	\Bigg)^{n's}\\
	&\le
	\sum_s\mathbf{1}\Bigg\{
	\f{(\log n')^2}{n'}
	\le 
	s \le \f{k^2}{2^k}
	\Bigg\}
	\binom{n}{ns}
	\f{s^{ns}}{ 2^{k(1-8\EPSCONST)ns} }
	\le
	\f1{\exp\{ \Omega( (\log n)^2) \}}\,.
	\end{align*} 
It follows that $r'(\cD)=o_n(1)$ with high probability over $\cD$, and the result follows.
\end{proof}

\section{Contraction estimates}\label{s:contract}

\noindent In this section we state and prove three key technical results:
\begin{enumerate}[--]
\item Proposition~\ref{p:nice.tree.lagrange}
considers a \bemph{nice} (Definition~\ref{d:nice}) subtree $T$ of a compound region $U$,
and analyzes the maximal-entropy judicious measure on $T$ subject to edge marginals $\omega_{\delta T}$ on the boundary edges $\delta T$. 

\item Proposition~\ref{p:contraction.for.simple.var}
considers the depth-one neighborhood of a non-compound variable $U$, and analyzes the maximal-entropy judicious measure on $U$ 
subject to edge marginals $\omega_{\delta U}$ on the boundary edges $\delta U$. 

\item Proposition~\ref{p:pair.clause.weights} shows how to reweight a clause to achieve a desired set of outgoing \textsc{bp} messages.
\end{enumerate}
The significance of these results in the proof outline is as follows:
\begin{enumerate}[--]
\item Proposition~\ref{p:contraction.for.simple.var}
solves the optimization problem for non-compound variables that was derived in Proposition~\ref{p:block.update.non.compound}. The main technical difficulty of this result is that the clause types around a non-compound variable are not fixed. 

\item On the other hand, in Section~\ref{s:merge} we will use
Propositions~\ref{p:nice.tree.lagrange}
and \ref{p:pair.clause.weights}
to prove Proposition~\ref{p:contraction.COMPOUND},
which analyzes
the maximal-entropy judicious measure on all of $U$ subject to boundary conditions $\omega_{\delta U}$.
This completes the solution of the optimization for compound regions that was derived in Proposition~\ref{p:update.compound}. The main technical difficulty of this result is to deal with \bemph{non-nice} regions of the compound enclosure, which are not covered by Proposition~\ref{p:nice.tree.lagrange}. At the same time, we are helped in the analysis by the notion of compound type (Definition~\ref{d:cpd.type}), which ensures that in the interior of a compound enclosure, the clause types around a variable are fixed.

\end{enumerate}
A more detailed outline of this section is given below, before the start of \S\ref{ss:lagrange.iteration}. 
We begin by specifying the form of the subtrees $T$ that we will consider for Proposition~\ref{p:nice.tree.lagrange}. 

\begin{dfn}[entropy maximization on rooted trees] \label{d:nu.judicious} Let $T$ be a finite tree rooted at a clause $\crt$ with exactly one child, which we refer to as the root variable $\vrt$. We use $\Leaves T$ to denote the set of all leaf vertices of $T$ other than $\crt$, and we assume $\Leaves T$ consists of variables only. We use $\delta T$ to denote the edges incident to $\Leaves T$. From now on we will refer to $\Leaves T$ and $\delta T$ respectively as the ``boundary variables'' and ``boundary edges'' of $T$. 
A small example appears in Figure~\ref{f:nice.subtree.of.compound.enclosure}. 
In the applications of this definition
in the analysis of Section~\ref{s:merge}, we will take $T$
to be a nice (Definition~\ref{d:nice}) subtree of a compound enclosure $U$. As a result, the notations that follow are purposefully similar to those of Definition~\ref{d:constrained.opt.compound.enclosure}. 
Let 
 	\[
	\Simplex(T)
	\equiv \left\{
	\hspace{-3pt}\begin{array}{c}
	\textup{probability measures $\nu$ over}\\
	\textup{pair colorings $(\usi^1,\usi^2)$
		of $T$}
	\end{array}\hspace{-3pt}\right\}\,.
	\]
We say that $\nu\in\Simplex(T)$ is \bemph{judicious}
if all of its edge marginals match the canonical marginals of Definition~\ref{d:canonical}: that is, for all edges $e$ in $T$ and for both $j=1,2$, 
we have $\nu((\sigma_e)^j=\sigma)=\starpi_e(\sigma)$
for all $\sigma\in\set{\RYGB}$. Let $\omega_{\delta T}$ denote a tuple $(\omega_e)_{e\in\delta T}$ where each $\omega_e$ is a judicious probability measure on $\set{\RYGB}^2$. We then let
	\[
	\Judicious(T ; \omega_{\delta T})
	\equiv
	\left\{
	\hspace{-3pt}\begin{array}{c}
	\nu\in\Simplex(T):
	\textup{$\nu$ is judicious,}\\
	\textup{and
	$\nu_e=\omega_e$ for all $e\in\delta T$}
	\end{array}\hspace{-3pt}
	\right\}\,,
	\]
where $\nu_e$ denotes the marginal of $\nu$ on edge $e$. Define the constrained optimizer
	\beq\label{e:constrained.opt}
	\optnu(T;\omega_{\delta T})
	\equiv \argmax_\nu
		\bigg\{ \Ent(\nu)
		: \nu \in
		\Judicious(T,\omega_{\delta T})\bigg\}\,.
	\eeq
Given weights $\Lm$ on $T$, we also define the unconstrained optimizer
	\beq\label{e:weighted.unconstrained.opt}
	\nu[T;\Lm]
	\equiv \argmax_\nu
		\bigg\{ \Ent(\nu)
			+ \langle \log\Lm,\nu\rangle
		: \nu \in \Simplex(T)\bigg\}\,,
	\eeq
where $\langle\log\Lm,\nu\rangle$ denotes the 
expected value of $\log\Lm(\usi)$ with $\usi$ distributed according to $\nu$. Note that by elementary calculus, the solution
$\nu=\nu[T;\Lm]$ of
\eqref{e:weighted.unconstrained.opt} is given explicitly by
	\beq\label{e:Lm.weighted.msr}
	\nu(\usi)
	= \f{\Ind{\textup{$\usi$
		is a valid coloring of $T$}
		} \Lm(\usi)}{\bm{z}[T;\Lm]}
	\eeq
where $\bm{z}[T;\Lm]$ is the normalizing constant. Finally, we say that $\Lm$ are \bemph{Lagrangian weights} (for the constrained optimization problem $\optnu(T;\omega_{\delta T}$ in \eqref{e:constrained.opt}) if $\langle\log\Lm,\nu\rangle$ is constant over $\nu\in\Judicious(T;\omega_{\delta T})$. We will parametrize Lagrangian weights in a particular way, described in Definition~\ref{d:param.weights.Lambda} below.
\end{dfn}

\begin{dfn}[error notation for edge distributions]
\label{d:error.notation.edge}
On any edge $e$, denote $\prodom_e\equiv\starpi_e\otimes\starpi_e$ where $\starpi_e$ is the canonical marginal on edge $e$ from Definition~\ref{d:canonical}. For any other probability measure $\omega_e$ on $\set{\RYGB}^2$, we define the \bemph{discrepancy} of $\omega_e$ (relative to $\prodom_e$) by
	\beq\label{e:def.discrepancy.measure.e}
	\disc(\omega_e)
	\equiv \disc_e(\omega)
	\equiv\sum_{\sigma\in\set{\RYGB}^2}
		(\vth)^{\Ind{ \red \in\set{\sigma^1,\sigma^2} }}
		\Bigg| \f{\omega_e(\sigma)}{\prodom_e(\sigma)}
			-1
		\Bigg|\,.\eeq
where $\vth\equiv 2^{-k\zeta/6}$
for an absolute constant $0<\zeta\le1/20$.
We shall ultimately take $\zeta=\ZETA/4$ where $\ZETA$ is the absolute constant in Proposition~\ref{p:apriori} below.
\end{dfn}

\begin{figure}[h]
\includegraphics{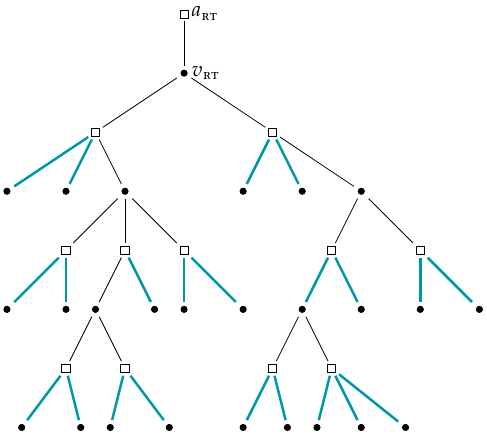}
\caption{A small example of a tree $T$ as in Definition~\ref{d:nu.judicious}. The edges of $\delta T$ are shown as thick colored lines. In the applications of this definition in Section~\ref{s:merge}, $T$ will be a nice subtree of a compound enclosure $U$.}
\label{f:nice.subtree.of.compound.enclosure}
\end{figure}

\begin{ppn}[contraction result inside compound enclosures]
\label{p:nice.tree.lagrange}
Assume that $R\ge k$, where $R$ is the neighborhood radius in \eqref{e:radii}. Let $U$ be a compound enclosure (Definition~\ref{d:enclosure}). Take $T\subseteq U$ of the form described in Definition~\ref{d:nu.judicious}, and suppose $T$ contains no defective variables. Consider the constrained optimization problem 
$\optnu(T;\omega_{\delta T})$ from
\eqref{e:constrained.opt},
where we assume that $\omega_{\delta T}$ satisfies,
for an absolute constant $0<\zeta\le1/20$,
the bounds
	\beq\label{e:apriori.assump.omega}
	\max_{\sigma\in\set{\RYGB}^2}
	\Bigg\{
	\f1{(2^k)^{\Ind{\sigma=\red\red}}}
	\bigg|
		\f{\omega_e(\sigma)}
		{\prodom_e(\sigma)}
		- 1 \bigg|
		\Bigg\} \le
		\f1{2^{2k\zeta}}
	\eeq
for all $e\in\delta T$. Then the following hold:
\begin{enumerate}[a.]
\item We can explicitly construct Lagrangian weights $\Lm\equiv \Lm(T;\omega_{\delta T})$ such that
the constrained optimizer $\nu=\optnu(T;\omega_{\delta T})$ of \eqref{e:constrained.opt} coincides with the unconstrained optimizer $\nu=\nu[T;\Lm]$ of \eqref{e:weighted.unconstrained.opt} (where we recall that the latter is simply the $\Lm$-weighted Gibbs measure on $T$).

\item Let $q$ be the solution of the $\Lm$-weighted \textsc{bp} recursions on $T$, where we fix the message ``from the root'' to be
	\[
		\Big(\hqstar_{\crt\vrt}
		\otimes
		\hqstar_{\crt\vrt}\Big)(\sigma)
		\equiv
		\prod_{j=1,2}
		\hqstar_{\crt\vrt}(\sigma^j)
		\,.
		\]
Then $\nu_e \cong \dq_e \hq_e$ for each edge $e$ in $T$.

\item For any edge $(av)$ in $T$, the
 discrepancy of $\nu_{av}$
(defined by \eqref{e:def.discrepancy.measure.e})
satisfies the bound
	\beq\label{e:clause.xi.prelim.defn}
	\disc_{av}(\nu)
	\le k^{O(1)}
	\underbrace{
	\sum_{e\in\delta T}
		\bigg(
		\f{(\vth)^{1/4}}{2^k}
		\bigg)^{\BTW_T(e,a)}
		\disc_e(\omega)
		}_{\xi_a(T;\omega_{\delta T})}\,,
	\eeq
where $\BTW_T(e,a)$ denotes the number of variables
between $e$ and $a$ on the unique path in $T$ that joins $e$ to $a$.
\end{enumerate}
The explicit parametrization of $\Lm$ is given in Definition~\ref{d:param.weights.Lambda} below, and the construction is given in 
Definition~\ref{d:nice.lagrange.defn.weights}.
\end{ppn}

\begin{ppn}[contraction result for non-compound variables]
\label{p:contraction.for.simple.var}
Assume that $R\ge k$, where $R$ is the neighborhood radius in \eqref{e:radii}.
In the same setting as Proposition~\ref{p:block.update.non.compound}, 
consider the constrained optimization problem
	\[
	\optnu(U;\omega_{\delta U})
	=\argmax_\nu
	\Bigg\{ \Ent(\nu)
		: \nu \in
		\Judicious_{\DD}(U,\omega_{\delta U})\Bigg\}\,,
	\]
where $U$ is the depth-one tree from
Definition~\ref{d:judicious.augmented.alphabet},
and we assume that (cf.\ \eqref{e:apriori.assump.omega})
	\beq\label{e:apriori.assump.omega.Augmented}
	\adjustlimits
	\max_{\bL}
	\max_{\sigma\in\set{\RYGB}^2}
	\Bigg\{
	\f1{(2^k)^{\Ind{\sigma=\red\red}}}
	\bigg|
		\f{\omega_{\bL,j(\bt_e)}(\sigma)}
		{\prodom_{\bL,j(\bt_e)}(\sigma)}
		- 1 \bigg|
		\Bigg\} \le
		\f1{2^{2k\zeta}}
	\eeq
for all $e\in\delta U$.
Then the following hold:
\begin{enumerate}[a.]
\item We can explicitly construct Lagrangian weights $\Psi\equiv\Psi_{\DD}(U;\omega_{\delta U})$ such that the solution $\nu=\optnu(U;\omega_{\delta U})$ of the above coincides with $\nu=\nu[U;\Psi]$, i.e., the $\Psi$-weighted Gibbs measure on $U$ for the augmented pair coloring model.

\item Let $q$ be the solution of the $\Psi$-weighted \textsc{bp} recursions on $U$. Then
$\nu_e\cong\dq_e\hq_e$ for each $e$ in $U$,
where $\nu_e$, $\dq_e$, and $\hq_e$ are probability measures on pairs $(\sigma,\bL)$ with $\sigma\in\set{\RYGB}^2$.

\item For the solution $\nu=\optnu(U;\omega_{\delta U})$
of the above, denote (cf.\ \eqref{e:def.discrepancy.measure.e})
	\beq\label{e:disc.augmented}
	\disc_e(\nu)\equiv
	\max_{\bL}
	\Bigg\{\sum_{\sigma\in\set{\RYGB}^2}
		(\vth)^{\Ind{ \red \in\set{\sigma^1,\sigma^2} }}
		\bigg| \f{\nu_e(\sigma\,|\,\bL)}
			{\prodom_e(\sigma)}
			-1
		\bigg|
		\Bigg\}
	\eeq
Then, for each edge $(av)$ incident to the root variable $v$ of $U$, we have
(cf.\ \eqref{e:clause.xi.prelim.defn})
	\beq\label{e:depth.one.update.final}
	\disc_{av}(\nu)
	\le k^{O(1)} 
	\underbrace{
	\sum_{e\in\delta U}
	\f{ (\vth)^{1/4}}{2^k}
	\disc_e(\omega_e)
	}_{\xi_a(U;\omega_{\delta U})}\,.
	\eeq
\end{enumerate}
The explicit parametrization of $\Psi$ is given in Definition~\ref{d:param.weights.Augmented} below, and the construction is given in Definition~\ref{d:lagrange.defn.Augmented}.
\end{ppn}

\begin{dfn}[error notation for clause-to-variable messages]
\label{d:clause.rel.error}
In the pair model, given any two functions $g,h:\set{\RYC}^2\to(0,\infty)$, we write
	\[
	\crelerr(g,h)
	=\begin{pmatrix}
	\displaystyle
	\max
	\bigg\{ \bigg|\f{h(\sigma)}{g(\sigma)}-1\bigg|
		: \red[\sigma]=0\bigg\} \smallskip\\
	\displaystyle
	\max
	\bigg\{ \bigg|\f{h(\sigma)}{g(\sigma)}-1\bigg|
		: \red[\sigma]=1\bigg\} \smallskip \\
	\displaystyle
	\max
	\bigg\{ \bigg|\f{h(\sigma)}{g(\sigma)}-1\bigg|
		: \red[\sigma]=2\bigg\}
	\end{pmatrix}
	\in\mathbb{R}^3.
	\]
We write $\crelerr(g,h) \le \vec{s}$
to mean that $\crelerr(g,h)$ is coordinatewise upper bounded by $\vec{s}$ in $\mathbb{R}^3$.
In most cases we use this notation when both $g$ and $h$ are clause-to-variable messages in the pair model.
\end{dfn}

Recall that in the single-copy model,
 given a variable-to-clause message $\dq_{va}$ (a measure on colors $\sigma\in\set{\RYGB}$), we defined a reweighted version $Q_{va}$ by \eqref{e:reweighted.messages.single}. We make the analogous definition in the pair model: given a variable-to-clause message $\dq\equiv\dq_{va}$ (a measure on pairs of colors $(\sigma^1,\sigma^2)\in\set{\RYGB}^2$), define its reweighted version
	\beq\label{e:reweighted.messages.v.to.c}
	Q_{va}(\sigma)
	= \f{\dq(\sigma)}
		{(2^{|\pd a|-1})^{\red[\sigma]}}
	\bigg/
	\bigg\{
	\dq(\set{\yel,\cya}^2)
	+\f{\dq(\set{\red}\times\set{\yel,\cya})
	+\dq(\set{\yel,\cya}\times\set{\red})}
		{2^{|\pd a|-1}}
	+\f{\dq(\red\red)}{4^{|\pd a|-1}}
	\bigg\}
	\,,\eeq
where we denote $\red[\sigma]\equiv\Ind{\sigma^1=\red}+\Ind{\sigma^2=\red}$ for $\sigma\in\set{\RYGB}^2$.

\begin{ppn}[pair version of Lemma~\ref{l:clause.bp.weights}]
\label{p:pair.clause.weights}
Suppose in the pair model that the clause $a$ receives incoming messages $\dq_e$ ($e\in\delta a$) whose reweighted versions $Q_e$ (defined by \eqref{e:reweighted.messages.v.to.c}) satisfy, for some absolute constant $0<\zeta\le1/20$,
	\beq\label{e:first.update.at.boundary}
	\max_{\sigma\in\set{\yel,\cya}^2}
	\bigg\{
	\bigg| Q_e(\sigma)-\f14\bigg| 
	\bigg\}\le \f{O(1)}{2^{k\zeta}}\,,\quad
	Q_e(\sigma)
	=\f{1 + O(2^{-k\zeta})}{4 \cdot 
		2^{|\pd a|-1}}
	\textup{ if }\red[\sigma]=1\,,\quad
	Q_e(\red\red) \le \f{O(1)}{2^{k(1+\zeta)}}\,,
	\eeq
Let $h\equiv\BP[\dq]$.
Meanwhile, let $\hq_{e,\infty}$ ($e\in\delta a$) 
denote a set of desired outgoing messages, and suppose that
	\[\crelerr( \hq_{e,\infty},h_e)
	\le \begin{pmatrix}
	\ep_e \\ \dot{\ep}_e \\ \ddot{\ep}_e
	\end{pmatrix}
	\le
	 \f1{k^4} \begin{pmatrix}
		1 \\ 1 \\1 \end{pmatrix}
	\]
for all $e\in\delta a$
(in the notation of Definition~\ref{d:clause.rel.error}). Then there exist clause weights $\Gm = (\gamma_e)_{e\in\delta a}$ such that the weighted \textsc{bp} recursion outputs
$\hq_{av,\infty}=\BP_{av}[\dq;\Gamma]$
for all $v\in\pd a$, and
these weights satisfy the bounds
	\[
	\crelerr(1,\gamma_e)
	\le
	O(1)
	\begin{pmatrix}
	\ep_e \\ 
	\ep_e+\dot{\ep}_e \\ 
	\ep_e+\ddot{\ep}_e
	\end{pmatrix}
	+
	O(k^6)
	\begin{pmatrix}
	2^{-k} & 2^{-k} & 2^{-k(1+\zeta)}\\
	1 & 2^{-k} & 2^{-k(1+\zeta)}\\
	1 & 2^{-k(1-\zeta)} & 2^{-k}
	\end{pmatrix}
		\sum_{e'\in\delta a}
	\begin{pmatrix}
	\ep_{e'}\\
	\dot{\ep}_{e'}\\
	\ddot{\ep}_{e'}
	\end{pmatrix}\]
for all $e\in\delta a$.\end{ppn}

\noindent
\textbf{Organization of the remainder of this section:}
\begin{enumerate}[--]
\item In \S\ref{ss:lagrange.iteration}
we give the precise construction
(see Definition~\ref{d:nice.lagrange.defn.weights}) of the weights $\Lm\equiv\Lm(T;\omega_{\delta T})$ of Proposition~\ref{p:nice.tree.lagrange}.
The construction iterates between a series of clause and variable updates, which we then proceed to analyze in the subsequent subsections, as follows:

\item In \S\ref{ss:clause.update}
we analyze the clause updates
in Definition~\ref{d:nice.lagrange.defn.weights}.
As a byproduct of the analysis we prove
Proposition~\ref{p:pair.clause.weights}.

\item In \S\ref{ss:var.update} 
we analyze the variable updates
in Definition~\ref{d:nice.lagrange.defn.weights}.
\item In \S\ref{ss:contraction.nondefect}
we combine the analysis
of \S\ref{ss:clause.update}
and \S\ref{ss:var.update} 
to prove Proposition~\ref{p:nice.tree.lagrange}.
	
\item In \S\ref{ss:ctypes.contract}
we give the analysis in the non-compound case
to prove Proposition~\ref{p:contraction.for.simple.var}.
\end{enumerate}

\subsection{Lagrange multipliers for subtrees of compound enclosures}
\label{ss:lagrange.iteration}

Fix $T$ as in Definition~\ref{d:nu.judicious}.
For the constrained optimization problem
$\optnu(T;\omega_{\delta T})$ of \eqref{e:constrained.opt},
we will always set boundary conditions $\omega_{\delta T}$ such that there is a unique solution
in the interior of the feasible domain of measures $\nu$. It then follows by general theory that 
there is a set of weights 
$\Lm \equiv \Lm(T;\omega_{\delta T})$ that are Lagrangian (recalling Definition~\ref{d:nu.judicious}, this means that $\langle\log\Lm,\nu\rangle$ is constant over $\nu\in\Judicious(T;\omega_{\delta T})$), and such that the solutions of \eqref{e:constrained.opt}
and \eqref{e:weighted.unconstrained.opt} coincide
--- i.e., such that
	\beq\label{e:lambda.lagrange}
	\argmax_\nu
		\bigg\{ \Ent(\nu)
		: \nu \in
		\Judicious(T,\omega_{\delta T})\bigg\}
	= \argmax_\nu
		\bigg\{ \Ent(\nu)
			+ \langle \log\Lm,\nu\rangle
		: \nu \in \Simplex(T)\bigg\}\,.
	\eeq
In this subsection, we describe the explicit construction of these weights in the setting of Proposition~\ref{p:nice.tree.lagrange}. We begin by remarking on a single-copy analogue which was obtained in previous sections.

\begin{rmk}\label{r:adjusted.bp.stability}Recall that in the single-copy model, if $U$ is a tree whose leaves are all variables and whose clauses are all strictly coherent, then Corollary~\ref{c:clause.bp.weights} guarantees the existence of weights $\Lmstar_U$ such that
the $\Lmstar_U$-weighted Gibbs measure on $U$ has edge marginals $\starpi$. If $U$ is moreover nice, then the weights are constructed explicitly by 
Corollary~\ref{c:clause.bp.weights.explicit}. Recall that the corresponding
 \textsc{bp} solution 
 is not given by $\qstar$, but rather by $\qbul$
 as defined by \eqref{e:redistributed.bp.messages}.
It follows from Corollary~\ref{c:clause.bp.weights.explicit}
that 
	\[
	\bigg\|
	\f{\dqbul_e}{\dqstar_e}-1\bigg\|_\infty 
	+\bigg\|
	\f{\hqbul_e}{\hqstar_e}-1\bigg\|_\infty 
	\le \f{O(1)}{k^{r/2}}\]
for all edges $e$ of $T$. If $R\ge k$ (as assumed in 
Proposition~\ref{p:nice.tree.lagrange}), then this error is very small. In particular, 
on nice edges it implies that $\qbul$ will satisfy the same estimates as $\qstar$ from Definition~\ref{d:nice}. For this reason, we assume $R\ge k$ in the rest of the section, even when not explicitly stated. From now on we denote
$\prodq\equiv\qbul\otimes\qbul$. In the setting of Proposition~\ref{p:nice.tree.lagrange},
we have a tree $T$ as in Definition~\ref{d:nu.judicious} 
or Figure~\ref{f:nice.subtree.of.compound.enclosure}. Let $T\setminus\crt$ denote the tree $T$ with $\crt$ removed: then all the leaves of $T\setminus\crt$ are variables, so $\Lmstar_{T\setminus\crt}$ is defined by
Corollaries~\ref{c:clause.bp.weights} and \ref{c:clause.bp.weights.explicit}. Then, for a single-copy coloring $\uta=\uta_T$ on $T$, we define
	\beq\label{e:Lmstar.T.short.defn}
	\Lmstar_T(\uta)
	\equiv
	\hqbul_{\crt\vrt}(\tau_{\crt\vrt})
	\cdot
	\Lmstar_{T\setminus\crt}(\uta_{T\setminus\crt})\,,
	\eeq
that is, we put an additional weight $\hqbul_{\crt\vrt}$ on the root clause $\crt$.
As a consequence, in the $\Lmstar_T$-weighted model on $T$, the \textsc{bp} message from $\crt$ to $\vrt$ will be precisely $\hqbul_{\crt\vrt}$.
\end{rmk}

\begin{dfn}[parametrization of Lagrangian weights $\Lm$
on a subtree of a compound enclosure]
\label{d:param.weights.Lambda}
Let $T$ be as in Definition~\ref{d:nu.judicious}.
We parametrize the Lagrangian weights $\Lm$ on $T$ as follows. 
Given a pair $\usi\equiv(\usi^1,\usi^2)$ of valid colorings of $T$, let $\ux\equiv(\ux^1,\ux^2)\equiv\ux(\usi)$ be the corresponding pair of frozen configurations on the variables of $T$. Recall from Remark~\ref{r:adjusted.bp.stability} that $\prodq=\qbul\otimes\qbul$. We then take the weights to be of form
	\[
	\Lm(\usi)
	= \prodhq_{\crt\vrt}(\sigma_{\crt\vrt})\Bigg\{
	 \prod_{v\in V_T\setminus \Leaves T}
	 \Lm_v(\usi_{\delta v})
	 \Bigg\}
	\Bigg\{\prod_{e\in\delta T}
		\lm_e(\sigma_e) \Bigg\}\,,
	\]
where for each internal variable $v\in V_T\setminus\Leaves T$ the weight factorizes as
	\beq\label{e:factorize.internal.weights.in.cpd}
	\Lm_v(\usi_{\delta v})
	\equiv
	\overbrace{\Bigg(\prod_{j=1,2}
	(\lm_v)^j((x_v)^j)\Bigg)}^{\lm_v(x_v)}
	\prod_{e\in\delta v}
	\overbrace{\Bigg(
	\prod_{j=1,2}
	(\lm_e)^j((\sigma_e)^j)
	\Bigg)}^{\lm_e(\sigma_e)}\,.
	\eeq
We assume that $(\lm_v)^j(\plus)=1$,
and $(\lm_e)^j((\sigma_e)^j)=1$
for all $(\sigma_e)^j\ne\red$;
we will often abbreviate
$(\lm_e)^j\equiv(\lm_e)^j(\red)$. If $\Lm$ is any weight on $T$ of the functional form just described, then $\Lm$ is a Lagrangian weight in the sense of Definition~\ref{d:nu.judicious}, meaning that $\langle\log\Lm,\nu\rangle$ is constant over $\nu\in\Judicious(T;\omega_{\delta T})$. Conversely, it is easy to see that any Lagrangian weight can be expressed in this form. Next, recalling the discussion of \S\ref{ss:bp}, there is a unique solution $q=\BPq(T;\Lm)$ for the $\Lm$-weighted \textsc{bp} recursions on $T$, obtained by recursing inwards from the leaves. In particular, for $e\in\delta T$ we have simply $\dq_e=\lm_e$, and at the root we have $\hq_{\crt\vrt}=\prodhq_{\crt\vrt}$. Analogously~to~\eqref{e:var.tuple.measure.weighted}, the marginals of $\nu=\nu[T;\Lm]$ (the $\Lm$-weighted measure, see \eqref{e:weighted.unconstrained.opt} and \eqref{e:Lm.weighted.msr}) can be expressed in terms of $q=\BPq(T;\Lm)$: for instance,
the marginal law of $\usi_{\delta v}$ under $\nu$ is given by 
	\beq\label{e:defn.nu.Lambda.q}
	\Big(
	\nu_{\delta v}[\Lm_v;\hq]
	\Big)(\usi_{\delta v}) 
	\equiv \f1{\dbz_v}
	\varphi_v(\usi_{\delta v})
	\Lm_v(\usi_{\delta v}) 
	\prod_{e\in\delta v}\hq_e(\sigma_e)
	\,,\eeq
where $\dbz_v$ is the normalization. (This holds even if $v\in \Leaves T$, provided we define $\Lm_v(\usi_{\delta v})\equiv\lm_e(\sigma_e)$ where $e$ is the one edge in $\delta v$.) Likewise, the marginal law on any edge $e$ of $T$ is given by $\nu_e(\sigma) \cong \dq_e(\sigma) \hq_e(\sigma)$ for $\sigma\in\set{\RYGB}^2$.
\end{dfn}

\begin{dfn}[iterative construction of Lagrangian weights on a subtree of a compound enclosure]
\label{d:nice.lagrange.defn.weights}
Continuing in the setting of Proposition~\ref{p:nice.tree.lagrange}, we now define a sequence of weights $\Lm_t$ (parametrized as in Definition~\ref{d:param.weights.Lambda} for each $t\ge0$), which will be proved in this section to converge as $t\to\infty$ to the Lagrangian weights $\Lm_\infty\equiv\Lm(T;\omega_{\delta T})$ of Proposition~\ref{p:nice.tree.lagrange}. At the same time we will define messages $q_t$ which will converge as $t\to\infty$ to the \textsc{bp} solution $q\equiv\BPq(T;\Lm_\infty)$ for the
$\Lm_\infty$-weighted model on $T$, as discussed above. We initialize our construction
at $t=0$ with weights
	\[
	\Lm_0(\usi)
	\equiv 
	\Big(\Lmstar_T\otimes\Lmstar_T\Big)(\usi)
	\equiv \prod_{j=1,2} \Lmstar_T(\usi^j)
	\]
where $\Lmstar_T$ are the single-copy weights defined by
\eqref{e:Lmstar.T.short.defn} in Remark~\ref{r:adjusted.bp.stability}. More explicitly, if $\uta$ denotes a single-copy coloring of $T$, then its weight under $\Lmstar_T$ is given by
	\beq\label{e:explicit.Lmstar.T}
	\Lmstar_T(\uta)
	=\hqbul_{\crt\vrt}(\sigma_{\crt\vrt}) 
	\Bigg\{
	\prod_{v\in V_T\setminus \Leaves T}
	\overbrace{
	\Bigg(
	\lmstar_v(x_v)
	\prod_{e\in\delta v}
	\lmstar_e(\tau_e)
	\Bigg)}^{\Lmstar_v(\uta_{\delta v})}
	\Bigg\}
	\Bigg\{
	\prod_{e\in\delta T}
	\lmstar_e(\tau_e)
	\Bigg\}\,,\eeq
where $x_v\in\set{\minus,\plus,\free}$ denotes the frozen spin corresponding to $\uta_{\delta v}$,
and we recall from 
 Corollary~\ref{c:clause.bp.weights}
 that for $e\in\delta T$ we take
 $\lmstar_e = \dqbul_e$ as defined by \eqref{e:redistributed.bp.messages}. Note that $\Lm_0$ is indeed of the functional form specified in Definition~\ref{d:param.weights.Lambda}. 
Let $q_0\equiv \BPq(T;\Lm_0)$;
explicitly, we have
$q_0=\prodq\equiv \qbul\otimes\qbul$
 (see Remark~\ref{r:adjusted.bp.stability}). 
For $t\ge1$, given $\Lm_{t-1}$ we define updated weights $\Lm_t$ (also of the functional form 
from Definition~\ref{d:param.weights.Lambda}) by making the following series of updates, started from the boundary $\delta T$ and working up to the root, then working back down to the boundary:
\begin{enumerate}[I.]
\item \bemph{Boundary updates.}
Recall from Definition~\ref{d:param.weights.Lambda} that for each boundary edge $e=(av)\in\delta T$, the weight $\lm_e$ is equivalent to the variable-to-clause message $\dq_e$. We update these weights by setting
	\[
	\lm_{av,t}(\sigma)
	\equiv \dq_{va,t}(\sigma)
	\equiv 
	\f{\omega_{av}(\sigma)}
		{\hq_{av,t-1}(\sigma)}
	\Bigg/\Bigg(
	\sum_\tau
	\f{\omega_{av}(\tau)}{\hq_{av,t-1}(\tau)}
	\Bigg)\,.
	\]

\item \bemph{Upward pass (from boundary to root).}
Having updated the boundary messages,
go up the tree,
alternating steps (a) and (b),
starting with step (a) at the clauses incident to $\delta T$:
\begin{enumerate}
\item At a clause $a\ne\crt$, suppose the upward messages $\dq_{ua}$ from the child variables $u$ have just been updated to their $t$-versions, while the downward message $\hq_{va}$ from the parent variable $v$ is still at its $(t-1)$-version. 
Then apply \textsc{bp} to update the upward message from the clause,
	\[
	\hq_{av,t}
	=\BP_{av}[\dq_t].
	\]
(Note that we have not yet defined $\dq_t$ on all edges of $T$, but the right-hand side is well-defined since it depends only on the $\dq_t$ messages coming from the child variables $u$ of $a$.)

\item At an internal variable $v$,
suppose the upward messages
$\hq_{bv}$ from the child clauses
$b$ have just been updated to their $t$-versions; while the downward message
$\hq_{av}$ from the parent clause $a$,
as well as the variable weight $\Lm_v$,
are still at their $(t-1)$-versions.
Let $\Lm_{v,t}$ be the unique choice of weights,
 of the functional form
\eqref{e:factorize.internal.weights.in.cpd} from
Definition~\ref{d:param.weights.Lambda}, such that the measure 
	\beq\label{e:var.update.II.b}
	\nu_{\delta v, t}(\usi_{\delta v}) 
	\equiv
	\f{1}{\dbz_{v,t}}
	\varphi_v(\usi_{\delta v})
	\Lm_{v,t}(\usi_{\delta v})
	\Bigg\{
	\hq_{av,t-1}(\sigma_{av})
	\prod_{b \in \pd v\setminus a}
	\hq_{bv,t}(\sigma_{bv})
	\Bigg\}
	\eeq
is judicious (in the sense of Definition~\ref{d:nu.judicious}).
For now we assume that the weights $\Lm_{v,t}$ exist;
we will explicitly construct them in Proposition~\ref{p:var.update} below. After obtaining $\Lm_{v,t}$, we apply \textsc{bp} to update the message upward from variable $v$, that is to say, we let
	\[
	\dq_{va,t}
	=\BP_{va}[\hq_t;\Lm_{v,t}]\,.
	\]
(Note again that we have not yet defined $\hq_t$ on all edges of $T$, but the right-hand side is well-defined since it depends only on the $\hq_t$ messages coming from the child clauses $b$ of $v$.)
\end{enumerate}
The upward pass is completed once we have applied step (2b) at the root $\vrt$.

\item \bemph{Downward pass (from root to boundary).} Recall from Definition~\ref{d:param.weights.Lambda} that we always put weight $\prodhq_{\crt\vrt}$ on the root clause $\crt$, and this will also be the \textsc{bp} message from $\crt$ to $\vrt$. We therefore start the downward pass with the trivial update
	\beq\label{e:message.from.top.always.product}
	\hq_{\crt\vrt,t}
	=\hq_{\crt\vrt,t-1}
	=\hq_{\crt\vrt,0}
	= \prodhq_{\crt\vrt}\,.
	\eeq
Then continue down the tree,
alternating steps (a) and (b),
starting from step (a) at $\vrt$:
\begin{enumerate}
\item At an internal variable $v$,
suppose that the downward message $\hq_{bv}$
from the parent clause $b$ has just been updated to its $t$-version. This means the upward messages
$\hq_{av}$ from the child clauses $a$,
as well as the variable weight $\Lm_v$,
were already updated to their $t$-versions during the preceding upward pass. Now apply \textsc{bp} to update the downward messages from the variable,
	\[
	\dq_{va,t}
	=\BP_{va}\bigg[
		\Big( \hq_{bv,t-1};
		(\hq_{cv,t}
		:c\in\pd v\setminus\set{a,b})
		\Big) ;\Lm_{v,t}
		\bigg]
	\]
where we have chosen to use the
$(t-1)$-version of the downward message $\hq_{bv}$
rather than the $t$-version.
This is merely convenient for our analysis, but will have no effect on the limit since the messages all converge.

\item At a clause $a\ne\crt$,
suppose the downward message $\dq_{ua}$
from the parent variable $u$
has just been updated to its $t$-version.
This means the upward messages
$\dq_{va}$ from the child variables $v$
were already updated to their $t$-versions
during the preceding upward pass. 
Apply \textsc{bp} to update the downward messages from the clause, again using
for convenience the $(t-1)$-version
of the message from above: 
	\[
	\hq_{av,t}
	=\BP_{av}\bigg[
		\Big(\dq_{ua,t-1},
		(\dq_{wa,t}:
		w\in\pd a\setminus\set{u,v})
		\Big)
		\bigg].
	\]
\end{enumerate}
The downward pass is completed once we have applied step (3b) at each of the clauses incident to the boundary edges $\delta T$. This completes the definition of $\Lm_t$ and $q_t$.
\end{enumerate}
In summary, by iterating the above steps, we obtain
$(\Lm_t,q_t)$ for all $t\ge0$ on $T$.
 The next few subsections (\S\ref{ss:clause.update}--\ref{ss:contraction.nondefect}) are devoted to analyzing this iteration under the assumptions of Proposition~\ref{p:nice.tree.lagrange}.
\end{dfn}

\subsection{Analysis of clause \textsc{bp}
	recursion in pair model}
\label{ss:clause.update}

In this subsection we analyze the clause \textsc{bp} recursion in the pair model. The main result of the subsection is Proposition~\ref{p:clause.update},
which will be applied to give estimates for Steps II(a) and III(b) in Definition~\ref{d:nice.lagrange.defn.weights}.
At the end of the subsection we also give the proof of Proposition~\ref{p:pair.clause.weights}.
Recall that clauses do not distinguish between the $\SPIN{green}$ and $\SPIN{blue}$ colors, so in this subsection we shall work with the reduced alphabet $\set{\RYC}$ with $\cya$ as in \eqref{e:cyan.purple}.

\begin{dfn}[error notation for variable-to-clause messages]\label{d:var.rel.error}
In the pair model, given two variable-to-clause messages $p$ and $q$ (both probability measures on $\set{\RYC}^2$), let $P$ and $Q$ be their reweightings as defined by \eqref{e:reweighted.messages.v.to.c}.
For $j=1,2$ let $P^j$ and $Q^j$ be the marginals of $P$ and $Q$ on the $j$-th coordinate. We then write
	\[\vrelerr(p,q)
	=\begin{pmatrix}
	\displaystyle
	\max
	\bigg\{ \bigg|\f{Q(\sigma)}{P(\sigma)}-1\bigg|
		: \sigma\in\set{\RYC}^2,
			\red[\sigma]=0\bigg\} \smallskip\\
	\displaystyle
	\max
	\bigg\{ \bigg|\f{Q(\sigma)}{P(\sigma)}-1\bigg|
		: \sigma\in\set{\RYC}^2,
			\red[\sigma]=1\bigg\} \smallskip\\
	\displaystyle
	\max
	\bigg\{ \bigg|\f{Q(\sigma)}{P(\sigma)}-1\bigg|
		: \sigma\in\set{\RYC}^2,
			\red[\sigma]=2\bigg\} \smallskip\\
	\displaystyle
	\max\bigg\{
	\sum_{j=1,2}\bigg|\f{Q^j(\sigma)}
		{P^j(\sigma)}-1\bigg|
		: \sigma \in \set{\yel,\cya}
	\bigg\}\smallskip\\
	\displaystyle
	\sum_{j=1,2}\bigg|\f{Q^j(\red)}
		{P^j(\red)}-1\bigg|
	\end{pmatrix}
	\in\mathbb{R}^5\,.
	\]
We write $\vrelerr(p,q) \le \vec{s}$
to mean that $\vrelerr(p,q)$ is coordinatewise upper bounded by $\vec{s}\in\mathbb{R}^5$.
\end{dfn}

\begin{dfn}[subsets of clause colorings]
\label{d:subsets.of.clause.colorings}
Given a clause $a\in F$, we now fix some notation for various subsets of colorings of $\delta a$. Abbreviate $K = |\delta a| \in \set{k,k-1}$. Let
$\Whi \equiv \set{\yel,\cya}^{K-1}$
and $\Yel\equiv \set{\yel}^{K-1}$. Then let
	\begin{align*}
	\Red
		&\equiv
		\Bigg\{
		\usi\in\set{\red,\yel}^{K-1}:
		\Big|\Big\{j\in [K]:\sigma_j=\red\Big\}\Big|=1
		\Bigg\}\,,\\
	\Cya
		&\equiv
		\Bigg\{
		\usi\in\set{\yel,\cya}^{K-1}:
		\Big|\Big\{j\in [K]:\sigma_j=\cya\Big\}\Big|=1
		\Bigg\}\,,\\
	\Cya_{\ge j}
		&\equiv
		\Bigg\{
		\usi\in\set{\yel,\cya}^{K-1}:
		\Big|\Big\{j\in [K]:\sigma_j=\cya\Big\}\Big|
			\ge j
		\Bigg\}\,.
	\end{align*}
For any two subsets
$\SPIN{A}$ and $\SPIN{B}$ of $\set{\RYC}^{K-1}$, we 
will abbreviate $\SPIN{AB}\equiv \SPIN{A}\times\SPIN{B}$.
\end{dfn}

\begin{ppn}\label{p:clause.update}
Suppose the clause $a\in F$ receives two sets of incoming messages $p$ and $q$, with each message a probability measure over $\set{\RYC}^2$ satisfying the estimates~\eqref{e:first.update.at.boundary}, 
such that, in the notation of Definition~\ref{d:var.rel.error}, we have
	\[
	\vrelerr(p_e,q_e)
	\le\begin{pmatrix}
		\delta_e \\
		\dot{\delta}_e \\
		\ddot{\delta}_e \\
		\mdel_e \\
		\mdelred_e
		\end{pmatrix}
	\]
for each $e\in\delta a$.
Then the outgoing clause-to-variable messages $g\equiv\BP[p]$ and $h\equiv\BP[q]$ satisfy, for all $e'\in\delta a$,
	\beq\label{e:clause.output.rel.error.matrix}
	\crelerr(g_{e'},h_{e'})
	\le
	O(k^2)
	\begin{pmatrix}2^{-k(1+\zeta)}
	& 4^{-k} & 2^{-k(1+\zeta)} & 2^{-k} & 2^{-k}\\
	2^{-k} & 2^{-k} & 2^{-k(1+\zeta)}
		& 1 & 2^{-k} \\
	1 & 4^{-k} & 2^{-k(1+\zeta)}
		& 2^{-k} & 2^{-k}
	\end{pmatrix}
	\sum_{e\in\delta a \setminus e'}
	\begin{pmatrix}
		\delta_e\\
		\dot{\delta}_e\\
		\min\set{\ddot{\delta}_e,1}\\
		\mdel_e\\
		\mdelred_e
	\end{pmatrix}\,.
	\eeq

\begin{proof} 
Throughout the proof, we fix an edge $e'=(av)$ on which to calculate the outgoing \textsc{bp} messages,
and suppress $e'$ from the notation when possible. We abbreviate
	\[
	\begin{pmatrix}
	\bde \\
	\dbde \\ 
	\ddbde \\ 
	\mbde \\
	\mbdered
	\end{pmatrix}
	\equiv
	\sum_{e\in\delta a \setminus e'}
	\begin{pmatrix}
		\delta_e \\
		\dot{\delta}_e \\
		\min\set{\ddot{\delta}_e,1}	 \\
		\mdel_e \\
		\mdelred_e
		\end{pmatrix}\,.\smallskip
	\]
\bemph{Step 1. Rewriting of \textsc{bp} equations.}
For $e\in\delta a$ let $P_e,Q_e$ be the reweighted versions of $p_e,q_e$ defined by \eqref{e:reweighted.messages.v.to.c}. Let $\cX$ and $\cY$ be the (non-normalized) measures over configurations $\usi_{\delta a \setminus e'}$ defined by
	\begin{align*}
	\cX(\usi_{\delta a \setminus e'})
	&=\prod_{e\in\delta a \setminus e'}
		\bigg\{
		(2^{|\pd a|-1})^{\red[\sigma_e]}
		P_e(\sigma_e)\bigg\}\,,\\
	\cY(\usi_{\delta a \setminus e'})
	&=\prod_{e\in\delta a \setminus e'}
		\bigg\{
		(2^{|\pd a|-1})^{\red[\sigma_e]}
		Q_e(\sigma_e)
		\bigg\}\,.
	\end{align*}
Then the equations $g\equiv g_{av} \equiv \BP_{av}[p]$ and $h\equiv h_{av} \equiv \BP_{av}[q]$ can be rewritten as
	\begin{align*}
	g(\sigma)
	&= \f1{Z_g} \sum_{\usi_{\delta a}:\sigma_{e'}=\sigma}
		\hat{\varphi}_a(\usi_{\delta a})
		\cX(\usi_{\delta a\setminus e'})\,,\\
	h(\sigma)
	&=\f1{Z_h} \sum_{\usi_{\delta a}:\sigma_{e'}=\sigma}
		\hat{\varphi}_a(\usi_{\delta a})
		\cY(\usi_{\delta a\setminus e'})\,,
	\end{align*}
where $Z_g$ and $Z_h$ are the normalizing constants,
and $\hat{\varphi}_a$ denotes the factor 
for the pair coloring model, as in \eqref{e:pair.color.model.factors}.
In the remainder of the proof, we first obtain error bounds between $\cX$ and $\cY$, then use this to deduce error bounds between $g$ and $h$.\smallskip

\noindent\bemph{Step 2. Error bounds between $\cX$ and $\cY$.} Using the notation of Definition~\ref{d:subsets.of.clause.colorings}, we have
	\begin{align*}
	\cX(\Yel\Yel)
	&= \prod_{e\in\delta a \setminus e'}
		P_e(\yel\yel)
	= \prod_{e\in\delta a \setminus e'}
		\Bigg\{
		Q_e(\yel\yel)
		+ O(\delta_e)\Bigg\}
	=\cY(\Yel\Yel)
		+ O\Bigg(\f{\bde}{4^k}\Bigg)
	=\Theta\Bigg(\f1{4^k}\Bigg)\,,\\
	\cX(\Red\Yel)
	&=\sum_{e\in\delta a\setminus e'} \Bigg\{ 
		2^{|\pd a|-1}
		 Q_e(\red\yel)
			+ O(\dot{\delta}_e) \Bigg\}
		\prod_{e''\in\delta a
			\setminus\set{e,e'}} \Bigg\{
		Q_{e''}(\yel\yel)
		+O(\delta_{e''})\Bigg\}
	= \cY(\Red\Yel)
		+O\Bigg( \f{\dbde+k\bde}{4^k} \Bigg)
		=\Theta\Bigg(\f{k}{4^k}\Bigg)\,,\\
	\cX(\Cya\Yel)
	&=\sum_{e\in\delta a\setminus e'}
	\Bigg\{Q_e(\cya\yel) + O(\delta_e)\Bigg\}
	\prod_{e''\in\delta a\setminus\set{e,e'}} \Bigg\{
		Q_{e''}(\yel\yel)
		+O(\delta_{e''})\Bigg\}
	=\cY(\Cya\Yel)
	+O\Bigg( \f{k\bde}{4^k} \Bigg)
	=\Theta\Bigg(\f{k}{4^k}\Bigg),.
	\end{align*}
Recall for $j=1,2$ we use $(P_e)^j$ to denote the marginal of $P_e$ on the $j$-th copy. Then
	\begin{align*}
	\cX(\Yel\Whi)
	&=\prod_{e\in\delta a \setminus e'}
		\Bigg\{ (P_e)^1(\yel)
		- P_e(\yel\red)\Bigg\}
	= \prod_{e\in\delta a \setminus e'}
		\Bigg\{ (Q_e)^1(\yel)
			\Big( 1 + O(\mdel_e)\Big)
		- Q_e(\yel\red)
			\Big( 1 + O(\dot{\delta}_e)\Big)
		\Bigg\}
		\\
	&=\prod_{e\in\delta a \setminus e'}
		\Bigg\{ 
			(Q_e)^1(\yel)- Q_e(\yel\red)
			+ O\bigg(
			\mdel_e + \f{\dot{\delta}_e}{2^k}
			\bigg)
		\Bigg\}
	=\cY(\Yel\Whi)
	+O\Bigg(
		\f{\mbde}{2^k} + \f{\bm{\dot{\delta}}}{4^k}
	\Bigg)
	=\Theta\Bigg(\f1{2^k}\Bigg)\,,
	\end{align*}
where the transition to the second line uses
the estimates \eqref{e:first.update.at.boundary}. 
By similar calculations, we have
	\begin{align*}
	\cX(\Cya\Whi)
	&=
	\sum_{e\in\delta a \setminus e'}
	\Bigg\{
	(P_e)^1 (\cya) - P_e(\cya\red) \Bigg\}
	\prod_{e''\in\delta a \setminus \set{e,e'}} 
	\Bigg\{
	(P_e)^1(\yel)-P_e(\yel\red)
	\Bigg\}\\
	&=\cY(\Cya\Whi)
	+ O\Bigg(
	\f{k\mbde}{2^k} + \f{k\dbde}{4^k}
	\Bigg)
	 = \Theta\Bigg(\f{k}{2^k}\Bigg) \\
	\nonumber
	\cX(\Red\Whi)
	&=2^{|\pd a|-1}
	\sum_{e\in\delta a \setminus e'}
	\Bigg\{
	(P_e)^1 (\red) - P_e(\red\red) \Bigg\}
	\prod_{e''\in\delta a \setminus \set{e,e'}} 
	\Bigg\{
	(P_e)^1(\yel)-P_e(\yel\red)
	\Bigg\}
	\\
	&=
	\cY(\Red\Whi)
	+ O\Bigg(
	\f{k\mbde+\mbdered}{2^k}
	+\f{\ddbde}{2^{k(1+\zeta)}}
	+\f{k\dbde}{4^k}
	\Bigg)
	 = \Theta\Bigg(\f{k}{2^k}\Bigg)\,, \\
	\cX(\Whi\Whi)
	&=\prod_{e\in\delta a \setminus e'}
	\Bigg\{ 1
		-(Q_e)^1(\red)-(Q_e)^2(\red)
		+Q_e(\red\red)
		+O\bigg(
		\f{\mdelred_e}{2^k} + 
		\f{ \min\set{\ddot{\delta}_e,1}}
			{2^{k(1+\zeta)}}
		\bigg)
		\Bigg\}\\
	&= \cY(\Whi\Whi)
		+ O\Bigg( \f{\mbdered}{2^k}
	+ \f{\bm{\ddot{\delta}}}
		{ 2^{k(1+\zeta)}}
		\Bigg) = \Theta(1)\,.
	\end{align*}
In the last calculation, the term involving 
$\min\set{\ddot{\delta}_e,1}$ arises because on the one hand the definition of $\ddot{\delta}_e$ implies
	\[
	Q_e(\red\red)
	=P_e(\red\red)\bigg\{ 1 + O(\ddot{\delta}_e)\bigg\}\,,
	\]
and on the other hand we assume that both $p$ and $q$ satisfy \eqref{e:first.update.at.boundary} which implies that
	\[
	\max\bigg\{P_e(\red\red),Q_e(\red\red)\bigg\}
	\le \f{O(1)}{2^{k(1+\zeta)}}\,.
	\]
We also comment that we have kept track of explicit powers of $k$, but this is not an important point; one could also write the proof with factors $k^{O(1)}$ where $O(1)$ stays bounded throughout. Lastly we have
	\begin{align*}
	\cX(\Cya\Cya)
	&=\sum_{e\in\delta a \setminus e'}
		P_e(\cya\cya)
		\prod_{e''\in\delta a \setminus\set{e,e'}}
		P_{e''}(\yel\yel)
	+\sum_{e\in\delta a \setminus e'}
		\sum_{e''\in\delta a \setminus\set{e,e'}}
		P_e(\cya\yel)
		P_{e''}(\yel\cya)
		\prod_{e'''\in\delta a \setminus\set{e,e',e''}}
		P_{e'''}(\yel\yel)\\
	&=\cY(\Cya\Cya)
		+O\Bigg( \f{k^2\bde}{4^k} \Bigg)
	= \Theta\Bigg( \f{k^2}{4^k}\Bigg)
	\,,\\
	\cX(\Red\Cya)
	&=2^{|\pd a|-1}
		\Bigg\{
		\sum_{e\in\delta a \setminus e'}
		P_e(\red\cya)
		\prod_{e''\in\delta a \setminus\set{e,e'}}
		P_{e''}(\yel\yel)
	+
	\sum_{e\in\delta a \setminus e'}
		\sum_{e''\in\delta a \setminus\set{e,e'}}
		P_e(\red\yel)
		P_{e''}(\yel\cya)
		\prod_{e'''\in\delta a \setminus\set{e,e',e''}}
		P_{e'''}(\yel\yel)\Bigg\}\\
	&=\cY(\Red\Cya)
	+O\Bigg(
	\f{k^2\bde +k \dbde}{4^k}
	\Bigg)= \Theta\Bigg( \f{k^2}{4^k}\Bigg)\,,\\
	\cX(\Red\Red)
	&= 4^{|\pd a|-1}
	\Bigg\{
	\sum_{e\in\delta a\setminus e'}
		P_e(\red\red)
		\prod_{e''\in\delta a \setminus\set{e,e'}}
			P_{e''}(\yel\yel)
		+\sum_{e\in\delta a\setminus e'}
		\sum_{e''\in\delta a \setminus \set{e,e'}}
			P_e(\red\yel)
			P_{e''}(\yel\red)
			\prod_{e'''\in\delta a 
				\setminus\set{e,e',e''}}
			P_{e'''} (\yel\yel)\Bigg\}\\
	&= \cY(\Red\Red)
	+O\Bigg(
	\f{k\bde+\ddbde}{2^{k(1+\zeta)}}
	+ \f{k \dbde+ k^2\bde}{4^k}
	\Bigg)
	= \cY(\Red\Red)
	+O\Bigg(
	\f{k\bde + \ddbde}{2^{k(1+\zeta)}}
	+ \f{k \dbde}{4^k}
	\Bigg)
	= O\Bigg( \f{k}{2^{k(1+\zeta)}}\Bigg)\,.
	\smallskip
	\end{align*}

\noindent\bemph{Step 3. Error bounds between $g$ and $h$.} Recall that in Step 1 we expressed $g$ and $h$ in terms of $\cX$ and $\cY$. We now use the bounds from Step 2 on the error between $\cX$ and $\cY$ to deduce bounds on the error between $g$ and $h$. 
The two easiest bounds are
	\begin{align*}
	Z_g g(\red\red)
	&= \cX(\Yel\Yel)
	=\cY(\Yel\Yel)
		\Bigg\{ 1 + O(\bde)
		\Bigg\}
	= Z_h h(\red\red)
		\Bigg\{ 1 + O(\bde)
		\Bigg\}\,,\\
	Z_g g(\red\cya)
	&=\cX(\Yel\Cya_{\ge1})
	=\cX(\Yel\Whi)-\cX(\Yel\Yel)
	=Z_h h(\red\cya)
	\Bigg\{1
	+O\bigg(
		\mbde
		+\f{\bm{\dot{\delta}} + \bde}{2^k} \bigg)
	\Bigg\}\,.
	\end{align*}
The remaining estimates are similar, only slightly more complicated. We next have
	\begin{align*}
	Z_g g(\red\yel)
	&= \cX\Big(\Yel\times(\Red\cup\Cya_{\ge2}) \Big)
	= \cX(\Yel\Red)
	+ \cX(\Yel\Whi)
	-\cX(\Yel\Yel)
	-\cX(\Yel\Cya) 
	=Z_h h(\red\yel)
	\Bigg\{
	1 +O\bigg( \mbde + \f{ \dbde + k\bde}{2^k} \bigg)
	\Bigg\}
	\end{align*}
The calculations for the spins in $\set{\yel,\cya}^2$ are all similar:
	\begin{align*}
	Z_g g(\cya\cya)
	&= \cX\Big( (\Cya_{\ge1})^2\Big)
	= \cX(\Whi\Whi)
	- \cX(\Whi\Yel)- \cX(\Yel\Whi)+\cX(\Yel\Yel)
	= Z_h h(\cya\cya)
	\Bigg\{1 + O\bigg( 
		\f{ \mbde+\mbdered }{2^k}
			+
			\f{\ddbde}
				{2^{k(1+\zeta)}}
			+ \f{\bde + \dbde}{4^k }
			\bigg)\Bigg\}\,,\\
	Z_g g(\cya\yel)
	&= \cX\Big(
		\Cya_{\ge1} \times (\Red\cup\Cya_{\ge2})
		\Big)
	= \Bigg( \cX(\Whi\Red)-\cX(\Yel\Red)\Bigg)
		+\Bigg(
		\cX\Big((\Cya_{\ge1})^2\Big)
		-\cX(\Whi\Cya)+\cX(\Yel\Cya)
		\Bigg)\\
	&= Z_h h(\cya\yel)
		\Bigg\{ 1 + O\bigg(
		\f{k\mbde+\mbdered}{2^k}
		+\f{\ddbde}{2^{k(1+\zeta)}}
		+\f{k^2\bde+k\dbde}{4^k}
		\bigg)
			\Bigg\}\,,\\
	Z_g g(\yel\yel)
	&= \cX\Big((\Red\cup\Cya_{\ge2})^2\Big)
	= \cX(\Red\Red)
	+\bigg(
	\cX(\Red\Whi)
	-\cX(\Red\Yel)-\cX(\Red\Cya)
	\bigg)
	+\bigg(
	\cX(\Whi\Red)
	-\cX(\Yel\Red)-\cX(\Cya\Red)
	\bigg)\\
	&\qquad\qquad\qquad\qquad\qquad
	+
	\Bigg(
	\cX(\Cya_{\ge1}\Cya_{\ge1})
	-(\cX(\Cya\Whi)-\cX(\Cya\Yel))
	-(\cX(\Whi\Cya)-\cX(\Yel\Cya))
	+\cX(\Cya\Cya)
	\Bigg)\\
	&=Z_h h(\yel\yel)
	\Bigg\{1+O\bigg(
	\f{k\mbde+\mbdered}{2^k}
	+ \f{k\bde+\ddbde}{2^{k(1+\zeta)}}
	+ \f{k\dbde}{4^k}
	\bigg)\Bigg\}\,.
	\end{align*}
Combining these estimates gives
	\[
	Z_g=Z_h\Bigg\{1+O\bigg(
	\f{k\mbde+\mbdered}{2^k}
	+ \f{k\bde+\ddbde}{2^{k(1+\zeta)}}
	+ \f{k\dbde}{4^k}
	\bigg)\Bigg\}\,.
	\]
The claimed bounds directly follow.
\end{proof}
\end{ppn}

Let us briefly summarize where we are, with respect to the proof of Proposition~\ref{p:nice.tree.lagrange}. In \S\ref{ss:lagrange.iteration} we specified the construction of the weights $\Lm$ (Definition~\ref{d:nice.lagrange.defn.weights}). In the current subsection (\S\ref{ss:clause.update}) we made a general analysis of a \bemph{clause} update
(Proposition~\ref{p:clause.update}) which can be applied to estimate the effects of both Step II(a) and Step III(b). In the next subsection \S\ref{ss:var.update} we will 
make a general analysis of a \bemph{variable} update, which will be applied to estimate the effects of 
both Step II(b) and Step III(a). 
We will combine these estimates in \S\ref{ss:contraction.nondefect} to complete the proof of Proposition~\ref{p:nice.tree.lagrange}.

To conclude the current subsection, we momentarily digress from the plan just described, in order to complete the proof of Proposition~\ref{p:pair.clause.weights}, which 
which makes use of the preceding result Proposition~\ref{p:clause.update}.

\begin{proof}[Proof of Proposition~\ref{p:pair.clause.weights}]
We first define a sequence of weights $\Gm_t\equiv (\gm_{e,t} : e\in\delta a)$, then show that they converge in the limit $t\to\infty$ to the desired weights $\Gm_\infty=\Gm$. We emphasize that this index $t$ is \bemph{purely local to the proof of this proposition}, and is not the same as the $t$ that indexes the up-and-down passes in
Definition~\ref{d:nice.lagrange.defn.weights}. \smallskip

\noindent\bemph{Step 1. Iterative definition of weights.} Recall that the goal is to find weights $\Gm$ such that $\BP_{av}[\dq;\Gamma]=\hq_{av,\infty}$
for all $v\in\pd a$. We initialize $\Gm_0\equiv1$. 
For all $t\ge0$ let $h_t\equiv \BP[\dq;\Gamma_t]$
and $g_t\equiv \BP(\gm_t\dq)$, so that $h_t$ is proportional to $\gamma_t g_t$ --- in this notation, the goal is to have in the limit
$\gamma_\infty g_\infty\cong\hq_\infty$. We then recursively set
	\beq\label{e:clause.update.gamma.against.yy}
	\gamma_{t+1}(\sigma)
	=\gamma_t(\sigma) \cdot
	\f{\hq_\infty(\sigma)}
	{h_t(\sigma)}
	\bigg/\bigg(
	\f{\hq_\infty(\yel\yel)}
	{h_t(\yel\yel)}\bigg)
	\eeq
on all the edges of $\delta a$.
The rationale for this choice is that it keeps
$\gamma_t(\yel\yel)=1$ for all $t\ge0$, and gives us
	\[
	\gamma_{t+1} g_t
	=
	\f{\gamma_{t+1}}{\gamma_t} \gamma_t g_t
	\cong
	\f{\hq_\infty}{h_t} \gamma_t g_t
	\cong
	\f{\hq_\infty}{h_t} h_t
	\cong \hq_\infty\,,
	\]
(which is, clearly, compatible with our eventual goal of $\gamma_\infty g_\infty\cong \hq_\infty$).\smallskip

\noindent\bemph{Step 2. Analysis of weights.}
Suppose, in the notation of Definition~\ref{d:clause.rel.error}, that
for all $e\in\delta a$ we have
	\beq\label{e:clause.def.eps.t}
	\crelerr(\gamma_{e,t},\gamma_{e,t+1})
	= \begin{pmatrix}
		\ep_e(t)\\ \dot{\ep}_e(t) \\ \ddot{\ep}_e(t)
		\end{pmatrix}\,.\eeq
Recall that $\gamma_0\equiv1$
and $h_0\equiv \BP[\dq]\equiv h$, so at $t=0$ 
the assumptions give
	\beq\label{e:g.to.gamma.matrix}
	\crelerr(\gamma_{e,0},\gamma_{e,1})
	=
	\crelerr\bigg(1,
	\f{\hq_{e,\infty}}{h_e}
	\cdot
	\f{h_e(\yel\yel)}{\hq_{e,\infty}(\yel\yel)}
	\bigg)
	\le O(1)\begin{pmatrix}
	1 & 0 & 0 \\
	1 & 1 & 0\\
	1 & 0 & 1
	\end{pmatrix}
	\crelerr(\hq_{e,\infty},h_e)
	= O(1)
	\begin{pmatrix}
	\ep_e \\ 
	\ep_e+\dot{\ep}_e \\ 
	\ep_e+\ddot{\ep}_e
	\end{pmatrix}\,.
	\eeq
For all $t\ge0$, it follows from 
\eqref{e:clause.update.gamma.against.yy} and
 the relation $h_t\cong\gamma_t g_t$ that
	\begin{align}\nonumber
	\f{\gm_{t+2}(\sigma)}{\gm_{t+1}(\sigma)}
	&=\f{\hq_\infty(\sigma)/\hq_\infty(\yel\yel)}
	{h_{t+1}(\sigma)/h_{t+1}(\yel\yel)}
	=\overbrace{\Bigg\{
	\f{\gamma_t(\sigma)\hq_\infty(\sigma)/\hq_\infty(\yel\yel)}
	{\gamma_{t+1}(\sigma)
	h_t(\sigma)/h_t(\yel\yel)}
	\Bigg\}}^\textup{equals one
		by \eqref{e:clause.update.gamma.against.yy}}
	\f{\gamma_{t+1}(\sigma) h_t(\sigma)/h_t(\yel\yel)}{
		\gamma_t(\sigma)
		h_{t+1}(\sigma)/h_{t+1}(\yel\yel)}\\
	&=
	\f{[h_t(\sigma)/\gamma_t(\sigma)]
	/[h_t(\yel\yel)/\gamma_t(\yel\yel)]}
	{[h_{t+1}(\sigma)/\gamma_{t+1}(\sigma)]
	/[h_{t+1}(\yel\yel)/\gamma_{t+1}(\yel\yel)]}
	=	\f{ g_t(\sigma)
		/g_t(\yel\yel) }
		{ g_{t+1}(\sigma)
		/g_{t+1}(\yel\yel) }\,.
	\label{e:relate.err.gamma.to.err.g}
	\end{align}
(the second-to-last step uses that $\gamma_t(\yel\yel)\equiv1$). We can bound the error between $g_t=\BP[\gamma_t\dq]$ and $g_{t+1}=\BP[\gamma_{t+1}\dq]$ by applying Proposition~\ref{p:clause.update}, as follows: let $p_t$ denote the normalized measure corresponding to $\gamma_t\dq$, so that, in the notation of Definition~\ref{d:var.rel.error},
	\[
	\vrelerr(p_{e,t},p_{e,t+1})
	\le O(1) \begin{pmatrix}
	\ep_e(t) \\
	\dot{\ep}_e(t)\\
	\ddot{\ep}_e(t)\\
	\ep_e(t) + \dot{\ep}_e(t)/2^k\\
	\dot{\ep_e}(t) + \ddot{\ep}_e(t) /2^{k\zeta}
	\end{pmatrix}
	+O\Bigg( \ep_e(t) +
		\f{\dot{\ep}_e(t)}{2^k}
		 + \f{\ddot{\ep}_e(t)}{2^{k(1+\zeta)}}
		 \Bigg)
	\mathbf{1}_5\,
	\]
--- on the right-hand side, the first term comes from the error between the non-normalized measures $\gamma_t q$ and $\gamma_{t+1}q$, while the second term comes from the normalization. In matrix notation,
	\[\vrelerr(p_{e,t},p_{e,t+1})
	\le O(1) 
	\begin{pmatrix}
	1 & 2^{-k} & 2^{-k(1+\zeta)}\\
	1 & 1 & 2^{-k(1+\zeta)}\\
	1 & 2^{-k} & 1 \\
	1 & 2^{-k} & 2^{-k(1+\zeta)}\\
	1 & 1 & 2^{-k\zeta}
	\end{pmatrix}
	\begin{pmatrix}
	\ep_e(t) \\ \dot{\ep}_e(t) \\ \ddot{\ep}_e(t)
	\end{pmatrix}\,.
	\]
The message $p_0=\dq$ satisfies the bounds \eqref{e:first.update.at.boundary} by assumption.
It follows from 
\eqref{e:g.to.gamma.matrix}
that the message $p_1$ (again, the normalization of $\gamma_1\dq$) also satisfies
\eqref{e:first.update.at.boundary}, simply because the error between $\gamma_0$ and $\gamma_1$ is negligible. We will argue by induction that $p_t$ satisfies 
\eqref{e:first.update.at.boundary} 
for all $t\ge0$. First we note that by the definition~\eqref{e:clause.def.eps.t}
and by \eqref{e:relate.err.gamma.to.err.g}, we have
	\[\begin{pmatrix}
		\ep_e(t+1)\\ \dot{\ep}_e(t+1)
		\\ \ddot{\ep}_e(t+1)
		\end{pmatrix}
	=\crelerr(\gamma_{e,t+2},\gamma_{e,t+1})
	\le O(1)
	\begin{pmatrix}
	1 & 0 & 0 \\
	1 & 1 & 0\\
	1 & 0 & 1
	\end{pmatrix}
	\crelerr(g_{t+1},g_t)\,,\]
where the $3\times3$ matrix in the last expression is the same as the one in \eqref{e:g.to.gamma.matrix}.
Combining with Proposition~\ref{p:clause.update}
(more precisely, taking the product of the $3\times3$ matrix in the last display,
the $3\times5$ matrix in \eqref{e:clause.output.rel.error.matrix},
and the $5\times3$ matrix above), we find
	\beq\begin{pmatrix}
		\ep_e(t+1)\\ \dot{\ep}_e(t+1)
		\\ \ddot{\ep}_e(t+1)
		\end{pmatrix}
	\le 
	O(k^2)
	\begin{pmatrix}
	2^{-k} & 2^{-k} & 2^{-k(1+\zeta)}\\
	1 & 2^{-k} & 2^{-k(1+\zeta)}\\
	1 & 2^{-k} & 2^{-k(1+\zeta)}
	\end{pmatrix}
	\sum_{e'\in\delta a \setminus e}
	\begin{pmatrix}
		\ep_{e'}(t)\\ \dot{\ep}_{e'}(t) 
		\\ \ddot{\ep}_{e'}(t)
		\end{pmatrix}\,.
	\label{e:eps.t.to.next.step.matrix}
	\eeq
It is clear from this that
$p_t$ satisfies
\eqref{e:first.update.at.boundary} for all $t\ge0$, and moreover that the iteration defined in Step 1 converges as $t\to\infty$, with limiting weights $\gamma_\infty\equiv\gamma$. Quantitatively, it follows from \eqref{e:eps.t.to.next.step.matrix} that
	\[
	E(t+1)
	\equiv
	\begin{pmatrix}
	1 & 2^{-k/2}
	 & 2^{-k(1/2+\zeta)}
	 \end{pmatrix}
	\sum_{e\in\delta a}
	 \begin{pmatrix}
	 \ep_e(t+1)\\ \dot{\ep}_e(t+1)
	 	\\ \ddot{\ep}_e(t+1)
	 \end{pmatrix}
	\le
	\f{O(k^3) E(t)}{2^{k/2}}\,,
	\]
(the $3\times1$ row vector in the definition of $E(t)$ is chosen because it is an approximate left eigenvector of the $3\times3$ matrix in the last bound of \eqref{e:eps.t.to.next.step.matrix}). We shall bound
	\[
	\crelerr(1,\gamma_e)
	\le O(1)
	\sum_{t=0}^\infty
	\crelerr(\gamma_{e,t},
		\gamma_{e,t+1})
	\le O(1)\sum_{t=0}^1
	\begin{pmatrix}
	\ep_e(t)\\
	\dot{\ep}_e(t)\\
	\ddot{\ep}_e(t)
	\end{pmatrix}
	+O(1) \sum_{t=2}^\infty E(t)
	\begin{pmatrix}
	1 \\ 2^{k/2} \\ 2^{k(1/2+\zeta)}
	\end{pmatrix}\,.\]
On the right-hand side, the $t=0$ is bounded by 
\eqref{e:g.to.gamma.matrix},
and the $t=1$ term can be bounded by \eqref{e:eps.t.to.next.step.matrix}.
Thanks to the exponential decay of $E(t)$ in $t$, 
the sum over $t\ge2$ can be bounded by
	\begin{align*}
	\sum_{t=2}^\infty E(t)
	\begin{pmatrix}
	1 \\ 2^{k/2} \\ 2^{k(1/2+\zeta)}
	\end{pmatrix}
	&\le 
	\f{O(k^6)}{2^k}
	E(0)
	\begin{pmatrix}
	1 \\ 2^{k/2} \\ 2^{k(1/2+\zeta)}
	\end{pmatrix}
	\le
	\f{O(k^6)}{2^k}
	\begin{pmatrix}
	1 \\ 2^{k/2} \\ 2^{k(1/2+\zeta)}
	\end{pmatrix}
	\begin{pmatrix}
	1 & 2^{-k/2} & 2^{-k(1/2+\zeta)}
	\end{pmatrix}
	\sum_{e'\in\delta a}
	\begin{pmatrix}
	\ep_{e'}\\
	\ep_{e'}+\dot{\ep}_{e'}\\
	\ep_{e'}+\ddot{\ep}_{e'}
	\end{pmatrix}\\
	&=
	O(k^6)
	\begin{pmatrix}
	2^{-k} & 2^{-3k/2} & 2^{-k(3/2+\zeta)}\\
	2^{-k/2} & 2^{-k} & 2^{-k(1+\zeta)}\\
	2^{-k(1/2-\zeta)} & 2^{-k(1-\zeta)} & 2^{-k}\\
	\end{pmatrix}
	\sum_{e'\in\delta a}
	\begin{pmatrix}
	\ep_{e'}\\
	\ep_{e'}+\dot{\ep}_{e'}\\
	\ep_{e'}+\ddot{\ep}_{e'}
	\end{pmatrix}
	\end{align*}
Combining these bounds gives altogether
	\[
	\crelerr(1,\gamma_e)
	\le
	O(1)
	\begin{pmatrix}
	\ep_e \\ 
	\ep_e+\dot{\ep}_e \\ 
	\ep_e+\ddot{\ep}_e
	\end{pmatrix}
	+
	O(k^6)
	\begin{pmatrix}
	2^{-k} & 2^{-k} & 2^{-k(1+\zeta)}\\
	1 & 2^{-k} & 2^{-k(1+\zeta)}\\
	1 & 2^{-k(1-\zeta)} & 2^{-k}
	\end{pmatrix}
		\sum_{e'\in\delta a}
	\begin{pmatrix}
	\ep_{e'}\\
	\ep_{e'}+\dot{\ep}_{e'}\\
	\ep_{e'}+\ddot{\ep}_{e'}
	\end{pmatrix}\,,\]
from which the claimed result follows.
\end{proof}

\subsection{Analysis of variable \textsc{bp}
recursion in pair model}
\label{ss:var.update} 

In this subsection we analyze the variable \textsc{bp} recursion in the pair model. The analysis will be applied to give estimates for Steps II(b) and III(a) in Definition~\ref{d:nice.lagrange.defn.weights}. The main results of the current subsection are Proposition~\ref{p:var.update} and Corollary~\ref{c:var.update}, which are stated next.
The remaining results of the subsection are technical lemmas used in the proofs of Proposition~\ref{p:var.update} and Corollary~\ref{c:var.update}, which we briefly outline here:
\begin{enumerate}[--]
\item Lemma~\ref{l:variable.close.to.product}
shows that, under the conditions of Proposition~\ref{p:var.update} on the messages in coming to variable $v$,
the resulting law of the frozen spin $x_v\in\set{\minus,\plus,\free}^2$ is close to a product measure.
\item Lemma~\ref{l:transfer.error}
shows (using the result of
Lemma~\ref{l:variable.close.to.product}) that
that reweighting the frozen spin $(x_v)^1$
does not greatly affect the law of the
other spin $(x_v)^2$, and vice versa.

\item Lemma~\ref{l:perturb.red.weights}
shows that when we reweight all the edges $e\in\delta v$, the marginal law of $\sigma_e\in\set{\RYGB}^2$ is mainly affected by the weight on edge $e$ alone,
with a much smaller effect from the edges
$\delta v\setminus e$.

\item Lemma~\ref{l:marginal.error} gives a bound on the marginal error in variable-to-clause messages under certain conditions.
\end{enumerate}
The proof of Proposition~\ref{p:var.update} uses Lemmas~\ref{l:transfer.error}~and~\ref{l:perturb.red.weights}. Corollary~\ref{c:var.update} follows from Proposition~\ref{p:var.update} together with Lemma~\ref{l:marginal.error}. From this point on, rather than keeping track of explicit powers of $k$, we will simply write $k^{O(1)}$ to indicate $k$ raised to powers bounded by an absolute constant.

\begin{ppn}\label{p:var.update} Let $v$ be a nice variable with incoming messages $\hat{p}$ (in the pair coloring model). Recalling the notation of Definition~\ref{d:clause.rel.error}, assume that
for all $e\in\delta v$ we have 
	\beq\label{e:condition.on.hat.messages}
	\crelerr(\prodhq_e,\hat{p}_e)
	\le
	\f{k^{O(1)}}{2^{k\zeta}}
	\begin{pmatrix} 
	2^{-k}
	\\ 1 \\ 2^k
	\end{pmatrix}\,.\eeq
Suppose $v$ has weight $\Theta\equiv\Theta_v$, of the functional form \eqref{e:factorize.internal.weights.in.cpd} from Definition~\ref{d:param.weights.Lambda} (with $\theta$ in place of $\lm$). Assume that the measure $\mu\equiv\nu_{\delta v}[\Theta;\hat{p}]$ (defined using the notation of \eqref{e:defn.nu.Lambda.q}) is judicious in the sense of 
Definition~\ref{d:nu.judicious},
and that
	\beq\label{e:assume.variable.weights.close.to.one}
	\adjustlimits
	\sum_{j=1,2}\sum_{x\in\set{\minus,\plus,\free}}
	\Big|\theta^j(x)-1\Big| \le
	\f{k^{O(1)}}{2^{k\zeta}}\,,\quad
	\max_{e\in\delta v}\Bigg\{
	\sum_{j=1,2}
	\Big|(\theta_e)^j-1\Big|
	 \Bigg\}
	\le \f{ k^{O(1)}}{2^{k\zeta}}\,.
	\eeq
Let $\hq$ be another set of incoming messages,
also satisfying \eqref{e:condition.on.hat.messages}, such that for all $e\in\delta v$ we have
	\[
	\crelerr(\hat{p}_e,\hq_e)
	\le \begin{pmatrix} 
		\ep_e \\ \dot{\ep}_e
		\\ \ddot{\ep}_e\end{pmatrix}
	\le
	\f{k^{O(1)}}{2^{k\zeta}}
	\begin{pmatrix} 
	2^{-k}
	\\ 1 \\ 2^k
	\end{pmatrix}\,.
	\]
Then there exist weights $\Lm\equiv\Lm_v$
(also parametrized as 
\eqref{e:factorize.internal.weights.in.cpd} from
Definition~\ref{d:param.weights.Lambda}) 
such that
	\begin{align}
	\label{e:defn.of.err.v.notation}
	\sum_{j=1,2}
	\sum_{x\in\set{\minus,\free}}
	\Bigg|\f{\lm^j(x)}{\theta^j(x)}-1\Bigg|
	&\le k^{O(1)} \Bigg(
		\overbrace{
		\sum_{e'\in\delta v}\bigg\{
		\ep_{e'}
		+\f{\dot{\ep}_{e'}}{2^k}
		+\f{\min\set{\ddot{\ep}_{e'},1}}
		{2^{k(1+\zeta)}}
		\bigg\}
		}^{\textup{denote this }\err_v}
		\Bigg)\,,\\
	\sum_{j=1,2}
	\Bigg|
	\f{(\lm_e)^j}{(\theta_e)^j}-1
	\Bigg|
	&\le k^{O(1)}\Bigg(
	\dot{\ep}_e
	+\f{\min\set{\ddot{\ep}_e,1}}{2^{k\zeta}}
	+ \err_v
	\Bigg) \nonumber
	\end{align}
and $\nu_{\delta v}[\Lm;\hq]$ is again judicious
in the sense of Definition~\ref{d:nu.judicious}.
\end{ppn}

\begin{cor}\label{c:var.update}
In the setting of Proposition~\ref{p:var.update}, 
for $\dot{p}=\BP[\hat{p};\Theta]$
and $\dq=\BP[\hq;\Lm]$ we have 
	\[
	\vrelerr(\dot{p}_e,
		\dq_e)
	\le
	k^{O(1)}
	\begin{pmatrix}
	1\\1\\1\\ 
	2^{-k\zeta} \\
	2^{-k\zeta}
	\end{pmatrix} \err_v
	+k^{O(1)}
	\begin{pmatrix}
	0&0&0\\
	0&1&1\\
	0&1&1\\
	1&2^{-k}& 2^{-k}\\
	1&1& 1\\
	\end{pmatrix}
	\begin{pmatrix}
	\ep_e\\
	\dot{\ep}_e\\
	\min\set{\ddot{\ep}_e,1} / 2^{k\zeta}
	\end{pmatrix}\,,\]
for all $e\in\delta v$
(using the notation of Definition~\ref{d:var.rel.error}),
and with $\err_v$
as defined by \eqref{e:defn.of.err.v.notation}.
\end{cor}

\begin{lem}\label{l:variable.close.to.product}
Let $v$ be a nice variable with incoming messages $\hq$ (in the pair model)
satisfying condition \eqref{e:condition.on.hat.messages}.
Let $\nu$ be the resulting 
probability measure on the frozen spin $x_v\in\set{\minus,\plus,\free}^2$. 
Then $\nu$ is close to a product measure:
	\beq\label{e:frozen.product.approx}
	\nu(x)
	=\Bigg\{ 1 + O\bigg(\f{1}{2^{k\zeta}}\bigg)
		\Bigg\} \cdot
	\begin{cases}
	1/4 & \textup{for } x\in \set{\minus,\plus}^2\,,\\
	1/(4 \cdot 2^k)
		&\textup{for }x\in\set{\minus\free,\plus\free,
			\free\minus,\free\plus}\,, \\ 
	1/(4 \cdot 4^k)
		&\textup{for }
			x=\free\free
	\end{cases}\eeq
for an absolute constant $0<\zeta\le1/20$.

\begin{proof}
This is a fairly direct calculation. 
Recall that $v$ is a nice variable,
so by Remark~\ref{r:adjusted.bp.stability}, condition~\eqref{e:condition.on.hat.messages} also holds with $\hqstarstar_e\equiv \hqstar_e\otimes\hqstar_e$ in place of $\prodhq_e \equiv \hqbul_e\otimes\hqbul_e$. The degrees $|\delta v(\PM)|$ are constrained by Definition~\ref{d:nice}, as are the canonical messages $\hqstar$; recall moreover that $\hqstar_e(\yel)=\hqstar_e(\grn)=\hqstar_e(\blu)$. We will use these facts repeatedly in what follows.
Recall moreover from \eqref{e:cyan.purple} that we defined $\pur\equiv\SPIN{purple}\equiv\set{\red,\blu}$. \smallskip

\noindent\bemph{Step 1. Estimates in single-copy model.} Consider the variable $v$ with incoming messages $\hqstar$, and let $\pi$ denote the resulting law of the 
frozen spin $x_v\in\set{\minus,\plus,\free}$. Then, for a normalizing constant $z_\star$, we have
	\[ \pi(\free)
	=\f1{z_\star}
	\Bigg\{\prod_{e\in\delta v(\plus)}
	\hqstar_e(\grn)\Bigg\}\Bigg\{
	\prod_{e\in\delta v(\minus)}
	\hqstar_e(\grn)\Bigg\}
	=\f1{z_\star}
	\Bigg\{\prod_{e\in\delta v(\plus)}
	\hqstar_e(\blu)\Bigg\}\Bigg\{
	\prod_{e\in\delta v(\minus)}
	\hqstar_e(\blu)\Bigg\}\]
Then, using the conditions from Definition~\ref{d:nice}, we have
	\begin{align*}
	\f{\pi(\plus)}{\pi(\free)}
	&=\Bigg\{1 - O\bigg(\f1{2^k}\bigg)\Bigg\}
	\Bigg\{
	\prod_{e\in\delta v(\plus)}
	\f{\hqstar_e(\pur)}{\hqstar_e(\grn)}
	\Bigg\}\Bigg\{\prod_{e\in\delta v(\minus)}
	\f{\hqstar_e(\yel)}{\hqstar_e(\grn)}
	\Bigg\} \\
	&=\Bigg\{1 - O\bigg(\f1{2^k}\bigg)\Bigg\}
	\prod_{e\in\delta v(\plus)}
	\bigg( 1 + \f{1 + O(2^{-k/10})}
	{2^{|\pd a(e)|-1}} \bigg)
	=2^k
	\Bigg\{ 1 + O\bigg( \f{k}{2^{k/10}}\bigg)\Bigg\}\,,
	\end{align*}
where $a(e)$ refers to the clause incident to edge $e$. The same estimate holds for $\pi(\minus)$. It follows that
	\[
	\pi(x)
	= \Bigg\{ 1 + O\bigg( \f{k}{2^{k/10}}\bigg)\Bigg\}
	\times
	\begin{cases}
	1/2 &\textup{for }x\in\set{\minus,\plus}\,,\\
	1/2^{k+1}
		&\textup{for }x=\free\,.
	\end{cases}\smallskip
	\]

\noindent\bemph{Step 2. Estimates in pair model.}
Let $\nu_\star\equiv\pi\otimes\pi$. Let $z,z_\star$ be the normalizing constant such that
	\[
	\nu(\free\free)
	= \f1{z}
	\prod_{e\in\delta v}
	\hq_e(\grn\grn)\,,\quad
	\nu_\star(\free\free)
	=\f1{z_\star}
	\prod_{e\in\delta v}
	\hqstarstar_e(\grn\grn)\,.\]
Taking the ratio between the two and using the assumption \eqref{e:condition.on.hat.messages} gives
	\[
	\f{z\nu(\free\free)}{z_\star\nu_\star(\free\free)}
	=\prod_{e\in \delta v} \Bigg(
	1 + \f{k^{O(1)}}{2^{k(1+\zeta)}}\Bigg)
	=1 + O\bigg( \f{k^{O(1)}}{2^{k\zeta}}\bigg)\,.
	\]
Similarly, we use \eqref{e:condition.on.hat.messages} again to calculate that
	\begin{align*}
	z\nu(\plus\free)
	&= \Bigg\{
	\prod_{e\in\delta v(\plus)} \hq_e(\pur\grn)
	\Bigg\}\Bigg\{
	\prod_{e\in\delta v(\minus)} \hq_e(\yel\grn)
	\Bigg\}
	\Bigg\{
	1-\prod_{e\in\delta v(\plus)}
	\bigg(
	1-\f{\hq_e(\red\grn)}{\hq_e(\pur\grn)}
	\bigg)
	\Bigg\} \\
	&= \Bigg\{
	\prod_{e\in\delta v(\plus)} \hq_e(\pur\grn)
	\Bigg\}\Bigg\{
	\prod_{e\in\delta v(\minus)} \hq_e(\yel\grn)
	\Bigg\}
	\Bigg\{
	1-\prod_{e\in\delta v(\plus)}
	\bigg(
	1-\f{ 1 + O(k^{O(1)}/2^{k\zeta})
		}{2^{|\pd a(e)|-1}}
	\bigg)
	\Bigg\} \\
	&=
	\Bigg\{
	\prod_{e\in\delta v(\plus)} \hq_e(\pur\grn)
	\Bigg\}\Bigg\{
	\prod_{e\in\delta v(\minus)} \hq_e(\yel\grn)
	\Bigg\}
	\Bigg\{
	1 + O\bigg(\f1{2^k}\bigg)\Bigg\}
	=z_\star\nu_\star(\plus\free)
	\Bigg\{ 1 + O\bigg(\f{k^{O(1)}}{2^{k\zeta}}
	\bigg)\Bigg\}
	\,.
	\end{align*}
The same estimate holds for the spins in
$\set{\free\plus,\minus\free,\free\minus}$. A similar calculation gives
	\begin{align*}
	z\nu(\plus\plus)
	&= \Bigg\{
	\prod_{e\in\delta v(\plus)} \hq_e(\pur\pur)
	\Bigg\}\Bigg\{
	\prod_{e\in\delta v(\minus)} \hq_e(\yel\yel)
	\Bigg\}
	\Bigg\{
	1
	-\prod_{e\in\delta v(\plus)}
	\bigg(1-
	\f{\hq_e(\red\pur)}{\hq_e(\pur\pur)}
	\bigg)
	-\prod_{e\in\delta v(\plus)}
	\bigg(
	1-\f{\hq_e(\pur\red)}{\hq_e(\pur\pur)}
	\bigg)
	+\prod_{e\in\delta v(\plus)}
	\f{\hq_e(\blu\blu)}{\hq_e(\pur\pur)}
	\Bigg\} \\
	&=
	\Bigg\{
	\prod_{e\in\delta v(\plus)} \hq_e(\pur\pur)
	\Bigg\}\Bigg\{
	\prod_{e\in\delta v(\minus)} \hq_e(\yel\yel)
	\Bigg\}
	\Bigg\{ 1 + O\bigg(\f{k^{O(1)}}{2^{k\zeta}}
	\bigg)\Bigg\}
	= z_\star\nu_\star(\plus\plus)
	\Bigg\{ 1 + O\bigg(\f{k^{O(1)}}{2^{k\zeta}}
	\bigg)\Bigg\}\,,
	\end{align*}
and the same estimate holds for the other spins in $\set{\plus\minus,\minus\plus,\minus\minus}$. Therefore
	\beq\label{e:frozen.product.approx.pi}
	\nu(x)
	= \pi(x^1)\pi(x^2)
\Bigg\{ 1 + O\bigg(\f{k^{O(1)}}{2^{k\zeta}}\bigg)
		\Bigg\} \,,
	\eeq
and the claim follows.
\end{proof}
\end{lem}

\begin{lem}\label{l:transfer.error}
Let $v$ be a nice variable with incoming messages $\hq$ (in the pair model) satisfying \eqref{e:condition.on.hat.messages}.
Consider two weight functions 
$\lm=\lm^1\otimes\lm^2$ and $\theta=\theta^1\otimes\theta^2$ 
where $\lm^j,\theta^j:\set{\minus,\plus,\free}\to(0,\infty)$ with
	\beq\label{e:bound.on.frozen.weights}
	\adjustlimits
	\sum_{j=1,2}
	\sum_{x\in\set{\minus,\plus,\free}} 
	\Bigg( \Big|\lm^j(x)-1\Big|
	+\Big|\theta^j(x)-1\Big|\Bigg)
	\le \f{k^{O(1)}}{2^{k\zeta}}\,.\eeq
Let $\nu$ be the law of the frozen spin $x_v\in\set{\minus,\plus,\free}^2$ that results from the messages $\hq$ and the weights $\lm$,
and let $\mu$ be the law of $x_v$
that results from $\hq$ and $\theta$. Then for the marginal on the first copy we have
	\[
	\nu^1(x^1)
	=	
	\Bigg\{1+
	O\bigg(
		\f{k^{O(1)}\|\delta^2\|_\infty}{2^{k\zeta}}
		\bigg)\Bigg\}
	\bigg(\f{\lm^1\mu^1}{\theta^1}\bigg) (x^1)
	 \Bigg/
	 \Bigg(
	\sum_{y^1} \bigg(\f{\lm^1\mu^1}{\theta^1} \bigg)
			(y^1)\Bigg)
	\,,\quad
	\delta^j\equiv \f{\lm^j}{\theta^j}-1\,.
	\]
That is to say, $\nu^1$ is approximately equal to the reweighting of $\mu^1$ by $\lm^1/\theta^1$, 
and does not depend much on $\lm^2/\theta^2$.

\begin{proof}
Let $\barpi$ be the law of the frozen spin $x=x_v\in\set{\minus,\plus,\free}^2$ that results from
messages $\hq$ and weights $\lm^1\otimes\theta^2$. Then
	\[
	\barpi(x)
	\cong \f{\lm^1(x^1)}{\theta^1(x^1)}
		\mu(x)\,,
	\]
from which it follows that the marginal of $\barpi$ on the first copy is exactly
	\[
	\barpi^1(x^1)
	= \bigg(\f{\lm^1\mu^1}{\theta^1}\bigg) (x^1)
	 \Bigg/
	 \Bigg(
	\sum_{y^1} \bigg(\f{\lm^1\mu^1}{\theta^1} \bigg)
			(y^1)\Bigg)\,.
	\]
It remains to compare $\nu^1$ with $\barpi^1$. To this end, note it follows from 
Lemma~\ref{l:variable.close.to.product}
together with the assumption \eqref{e:bound.on.frozen.weights} that $\barpi$ also satisfies the estimates
\eqref{e:frozen.product.approx}
or equivalently
\eqref{e:frozen.product.approx.pi}. We thus have $e:\set{\minus,\plus,\free}^2\to\mathbb{R}$ such that
	\beq\label{e:frozen.product.approx.err}
	\barpi(x)
	=\pi(x^1)\pi(x^2)
	\bigg( 1 + e(x)\bigg)\,,\quad
	|e(x)| \le 
		 \f{k^{O(1)}}{2^{k\zeta}}\,.
	\eeq
Recall that we defined $\delta^j\equiv \lm^j/\theta^j-1$. Then, with $z$ a normalizing constant, we have
	\[
	\nu(x)
	=\f1z \f{\lm^2(x^2)}{\theta^2(x^2)}
	\barpi(x)
	=\f1z
	\Big(1 + \delta^2(x^2) \Big) \barpi(x)\,.
	\]
We can use \eqref{e:frozen.product.approx.err} to estimate the normalizing constant as
	\begin{align*}
	z
	&= \sum_x
	\Big(1 + \delta^2(x^2) \Big)
	\barpi(x)
	= 1 + \sum_x\delta^2(x^2) 
	\barpi(x)
	= 1 + \sum_{x^1}
		\pi(x^1)
		\sum_{x^2} \pi(x^2) \delta^2(x^2) 
		\Big(1+ e(x)\Big)\\
	&= \Bigg\{1 + \sum_{x^2} \pi(x^2) \delta^2(x^2) \Bigg\}
	\Bigg\{1
	+ O\bigg(
		\f{k^{O(1)}\|\delta^2\|_\infty}{2^{k\zeta}}
		\bigg)\Bigg\}\,.
	\end{align*}
Similarly, we can also use \eqref{e:frozen.product.approx.err} to estimate
	\begin{align*}
	\sum_{x^2}
	\Big(1 + \delta^2(x^2) \Big) \barpi(x)
	&=\barpi^1(x^1)
	+ \sum_{x^2}\delta^2(x^2) \barpi(x) 
	=\barpi^1(x^1)
	+ \pi(x^1)\sum_{x^2}\delta^2(x^2)
		\pi(x^2)\Big(1+e(x)\Big)\\
	&=\barpi^1(x^1)
	+ \pi(x^1) \sum_{x^2} \pi(x^2) \delta^2(x^2) 
	+ O\bigg(
		\f{k^{O(1)}\|\delta^2\|_\infty}{2^{k\zeta}}
		\bigg)\\
	&= \barpi^1(x^1)
	\Bigg\{
	1+\sum_{x^2} \pi(x^2) \delta^2(x^2) 
	\Bigg\}
	\Bigg\{1+
	O\bigg(
		\f{k^{O(1)}\|\delta^2\|_\infty}{2^{k\zeta}}
		\bigg)\Bigg\}\,.
	\end{align*}
Taking the ratio between the last two quantities gives
	\[
	\nu^1(x^1)
	=\barpi^1(x^1)
	\Bigg\{1+
	O\bigg(
		\f{k^{O(1)}\|\delta^2\|_\infty}{2^{k\zeta}}
		\bigg)\Bigg\}\,.
	\]
This proves the claim.
\end{proof}
\end{lem}

\begin{lem}\label{l:perturb.red.weights}
Let $v$ be a nice variable. Suppose we are given two sets of weights for $v$, $\Theta$ and $\Lm$, both 
 of the functional form \eqref{e:factorize.internal.weights.in.cpd} from
Definition~\ref{d:param.weights.Lambda}, satisfying the bound
\eqref{e:bound.on.frozen.weights} from Lemma~\ref{l:transfer.error}, and additionally such that
	\beq\label{e:cond.on.msg.for.edge.weight.perturb}
	\max_{e\in\delta v}\Bigg\{
	\sum_{j=1,2}
	\bigg(
	\Big|(\theta_e)^j-1\Big|
	+\Big|(\lm_e)^j-1\Big|
	\bigg)
	\Bigg\}
	\le \f{k^{O(1)}}{2^{k\zeta}}\,.
	\eeq
Assume that the weights on the frozen spin agree, i.e., $\lm^j(x)=\theta^j(x)$ for $j=1,2$
and all $x\in\set{\minus,\plus,\free}$, so that only the edge weights can differ betweeen $\Theta$ and $\Lm$. 
Denote the error in the edge weights by
	\[
	(\rho_e)^j
	\equiv
	\f{(\lm_e)^j}{(\theta_e)^j}-1
	\equiv
	\f{(\lm_e)^j(\red)}{(\theta_e)^j(\red)}-1\,.
	\]
Let $v$ have incoming messages $\hq$ (in the pair model) satisfying \eqref{e:condition.on.hat.messages}.
Let $\mu=\nu_{\delta v}[\Theta;\hq]$
and $\nu=\nu_{\delta v}[\Lm;\hq]$,
with edge marginals $\mu_e$ and $\nu_e$
for $e\in\delta v$. Then
it holds for all $e\in\delta v$ that
	\[\f{(\nu_e)^1(\red)}{(\nu_e)^1(\blu)}
	=
	\f{(\lm_e)^1}{(\theta_e)^1}
	\cdot
	\f{(\mu_e)^1(\red)}{(\mu_e)^1(\blu)}
	\Bigg\{
	1 + O\Bigg(
	k^{O(1)}\bigg[
	{\f{(\rho_e)^2}{2^{k\zeta}} }+
	\sum_{j=1,2}
		\sum_{e'\in\delta v \setminus e}
		\f{|(\rho_{e'})^j|}{2^{k(1+\zeta)}}
	\bigg]\Bigg)
	\Bigg\}\,.\]
This says that the marginal $\SPIN{red}$-to-$\SPIN{blue}$ ratio in the first copy is changed by a factor of approximately $(\lm_e)^1/(\theta_e)^1$
--- i.e., for the first-copy marginal on $e$, changing the edge weights in both copies on all of $\delta v$ has approximately the same effect as changing the edge weight in the first copy only, on $e$ alone.
The analogous bound holds exchanging the copy indices $j=1,2$.

\begin{proof} For $e\in\delta v$ 
let $g_e,h_e$ be the non-normalized measures
on $\sigma\in\set{\RYGB}^2$ defined by
	\begin{align*}
	g_e(\sigma)
	&\equiv \hq_e(\sigma)\theta_e(\sigma)
	=\hq_e(\sigma)
	\Bigg\{
	\prod_{j=1,2}
	( (\theta_e)^j)^{\Ind{\sigma^1=\red}}
	\Bigg\}\,,\\
	h_e(\sigma)
	&\equiv \hq_e(\sigma)\lm_e(\sigma)
	= g_e(\sigma)
	\Bigg\{
	\prod_{j=1,2}
	\bigg(
	 1 + ((\rho_e)^j)^{\Ind{\sigma^j=\red}}
	 \bigg)
	\Bigg\}\,.
	\end{align*}
Similarly as in \eqref{e:incident.edges.of.same.sign}, 
let $\delta v(\PM e)
\equiv \set{e'\in\delta v\setminus e : \lit_e=\PM \lit_{e'}}$. Then, for some normalizing constants $z_e,\bar{z}_e$ we have
	\begin{align*}
			\f{\bar{z}_e\mu_e(\red\red)}
		{\theta^1(\lit_e)\theta^2(\lit_e)}
	&=
	(\theta_e)^1(\theta_e)^2
	\hq_e(\red\red)
	\Bigg\{
	\prod_{e'\in\delta v(\plus e)}
	g_{e'}(\pur\pur)
	\Bigg\}\Bigg\{
	\prod_{e'\in\delta v(\minus e)}
	g_{e'}(\yel\yel)\Bigg\}
	\,,\\
	\f{z_e\nu_e(\red\red)}
		{\theta^1(\lit_e)\theta^2(\lit_e)}
	&=
	(\lm_e)^1 (\lm_e)^2 
	\hq_e(\red\red)
	\Bigg\{
	\prod_{e'\in\delta v(\plus e)}
	h_{e'}(\pur\pur)
	\Bigg\}\Bigg\{
	\prod_{e'\in\delta v(\minus e)}
	h_{e'}(\yel\yel)\Bigg\}
	\,.\end{align*}
Note that $g_{e'}(\yel\yel)=\hq_{e'}(\yel\yel)=h_{e'}(\yel\yel)$ for all $e'\in\delta v$. On the other hand, we can estimate
	\begin{align*}
	&\prod_{e'\in\delta v(\plus e)}
	\f{h_{e'}(\pur\pur)}
	{g_{e'}(\pur\pur)}
	=\prod_{e'\in\delta v(\plus e)}
	\Bigg\{
	1+
	\f{
	(\rho_{e'})^1 g_{e'}(\red\pur)
	+(\rho_{e'})^2 g_{e'}(\pur\red)
	+(\rho_{e'})^1(\rho_{e'})^2 g_{e'}(\red\red)
	}
	{g_{e'}(\pur\pur)}\Bigg\} \\
	&\quad
	\stackrel{\eqref{e:cond.on.msg.for.edge.weight.perturb}}{=}
	\prod_{e'\in\delta v(\plus e)}
	\Bigg\{
	1+
	\bigg(1 + \f{k^{O(1)}}{2^{k\zeta}}\bigg)
	\f{
	(\rho_{e'})^1 \hq_{e'}(\red\pur)
	+(\rho_{e'})^2 \hq_{e'}(\pur\red)
	}
	{ \hq_{e'}(\pur\pur)}
	+O\bigg(\f{(\rho_{e'})^1(\rho_{e'})^2}
		{ 2^{k(1+\zeta)} }
		\bigg)
	\Bigg\}\\
	&\quad
	\stackrel{\eqref{e:condition.on.hat.messages}}{=}
	1 +
	\sum_{j=1,2}\Bigg(
	\underbrace{\sum_{e'\in\delta v(\plus e)}
			\f{(\rho_{e'})^j \hqstar_{e'}(\red)}
				{\hqstar_{e'}(\blu)}
				}_{\textup{denote this }A^j(\plus)}
				\Bigg)
	+ O\Bigg(\underbrace{
	\sum_{j=1,2}\sum_{e'\in\delta v\setminus e}
	\f{k^{O(1)} |(\rho_{e'})^j|}{2^{k(1+\zeta)}}
	}_{\textup{denote this }R_0}
	\Bigg)
	\end{align*}
(having implicitly used Remark~\ref{r:adjusted.bp.stability} together with
\eqref{e:condition.on.hat.messages}). It follows that
	\[
	\f{z_e\nu_e(\red\red)}
		{\bar{z}_e\mu_e(\red\red)}
	\Bigg\{\prod_{j=1,2}
	\f{(\theta_e)^j}{(\lm_e)^j}
	\Bigg\}
	=
	1 + A^1(\plus)+ A^2(\plus) + O(R_0)\,.\]
By very similar calculations we obtain,
for all $\sigma\in\set{\red,\blu}^2$,
	\[
	\f{z_e\nu_e(\sigma)}
		{\bar{z}_e\mu_e(\sigma)}
	\Bigg\{
	\prod_{j=1,2}
	\bigg(
	\f{(\theta_e)^j}{(\lm_e)^j}
	\bigg)^{\Ind{\sigma^j=\red}}
	\Bigg\}
	=
	1 + A^1(\plus)+ A^2(\plus) + O(R_0)\,.\]
Let $A^j(\minus)$ be define analogously to
$A^j(\plus)$, except that we sum over $e'\in\delta v(\minus e)$ rather than 
$e'\in\delta v(\plus e)$. Then similar calculations as above give the estimates
	\begin{align*}
	\f{z_e\nu_e(\red\yel)}
		{\bar{z}_e\mu_e(\red\yel)}
	\f{(\theta_e)^1}{(\lm_e)^1}
	&= 
	1 + A^1(\plus)+ A^2(\minus) + O(R_0)
	=\f{z_e\nu_e(\blu\yel)}
		{\bar{z}_e\mu_e(\blu\yel)}\,,\\
	\f{z_e\nu_e(\red\grn)}
		{\bar{z}_e\mu_e(\red\grn)}
	\f{(\theta_e)^1}{(\lm_e)^1}
	&=
	1 + A^1(\plus) + O(R_0)
	=\f{z_e\nu_e(\blu\grn)}
		{\bar{z}_e\mu_e(\blu\grn)}\,.
	\end{align*}
Combining these estimates gives 
	\begin{align*}
	\f{(\nu_e)^1(\red)}
	{(\nu_e)^1(\blu)}
	&=\f{(\lm_e)^1}{(\theta_e)^1}
	\f{(\mu_e)^1(\red)[1+A^1(\plus)]
	+\mu_e(\red\pur)A^2(\plus)
	+\mu_e(\red\yel)A^2(\minus)
	+O(\mu_e(\red\red) (\rho_e)^2)}
	{(\mu_e)^1(\blu)[1+A^1(\plus)]
	+\mu_e(\blu\pur)A^2(\plus)
	+\mu_e(\blu\yel)A^2(\minus)
	+O(\mu_e(\blu\red)(\rho_e)^2)
	}\Bigg\{1+O(R_0)\Bigg\}\\
	&=
	\f{(\lm_e)^1}{(\theta_e)^1}
	\f{(\mu_e)^1(\red)}{(\mu_e)^1(\blu)}
	\f{\DS 1
	+ \f{\mu_e(\red\pur)A^2(\plus)}
		{(\mu_e)^1(\red)[1+A^1(\plus)]}
	+\f{\mu_e(\red\yel)A^2(\minus)}
		{(\mu_e)^1(\red)[1+A^1(\plus)]}
	+ O\bigg( \f{k^{O(1)}(\rho_e)^2}{2^{k\zeta}}\bigg)
	}{\DS 1
	+ \f{\mu_e(\blu\pur)A^2(\plus)}
		{(\mu_e)^1(\blu)[1+A^1(\plus)]}
	+\f{\mu_e(\blu\yel)A^2(\minus)}
		{(\mu_e)^1(\blu)[1+A^1(\plus)]}
	+ O\bigg( \f{(\rho_e)^2}{2^k}\bigg)
	}\Bigg\{1+O(R_0)\Bigg\}\\
	&=\f{(\lm_e)^1}{(\theta_e)^1}
	\f{(\mu_e)^1(\red)}{(\mu_e)^1(\blu)}
	\Bigg\{1+O\bigg(
	\f{k^{O(1)}(\rho_e)^2}{2^{k\zeta}}+R_0 
		\bigg)\Bigg\}\,.
	\end{align*}
In the last step we used the assumptions
 \eqref{e:condition.on.hat.messages},
 \eqref{e:bound.on.frozen.weights}, and \eqref{e:cond.on.msg.for.edge.weight.perturb}, which together guarantee that
	\begin{align*}
	\f{\mu_e(\red\pur)}
	{(\mu_e)^1(\red)}
	&=\starpi_e(\pur)\Bigg\{ 1 + O\bigg(
		\f{k^{O(1)}}{2^{k\zeta}}
		\bigg)\Bigg\}
	=\f{\mu_e(\blu\pur)}
	{(\mu_e)^1(\blu)}\,,\\
	\f{\mu_e(\red\yel)}
	{(\mu_e)^1(\red)}
	&=\starpi_e(\yel)\Bigg\{ 1 + O\bigg(
		\f{k^{O(1)}}{2^{k\zeta}}
		\bigg)\Bigg\}
	=\f{\mu_e(\blu\yel)}
	{(\mu_e)^1(\blu)}\,.
	\end{align*}
The claim follows.
\end{proof}
\end{lem}

\begin{proof}[Proof of Proposition~\ref{p:var.update}]
We first define a sequence of weights $\Lm_t$, then analyze the construction to show that they converge to the desired weights $\Lm_\infty=\Lm$. We emphasize that this index $t$ is \bemph{purely local to the proof of this proposition}, and is not the same as the $t$ that indexes the up-and-down passes in
Definition~\ref{d:nice.lagrange.defn.weights}. The proof below is divided into a few numbered parts.
\smallskip

\noindent\bemph{Part~1. Iterative definition of weights.} Initialize $\Lm_0=\Theta$. For all $t\ge0$, given the weights $\Lm_t$, let $\nu_t\equiv\nu_{\delta v}[\Lm_t;\hq]$. Then, for all $e\in\delta v$, update the edge weights by setting, for both $j=1,2$,
\beq\label{e:var.update.edgeupdate}
	(\lm_{e,t+1})^j
	=(\lm_{e,t})^j \cdot \f{
	 \starpi_e(\red) / \starpi_e(\blu)}
	{ (\nu_{e,t})^j (\red) / (\nu_{e,t})^j(\blu)}\,.
	\eeq
Let $\nu_{t+1/2}$
be the measure 
that results from
incoming messages $\hq$,
frozen spin weights $\lm_t$,
and edge weights $\lm_{e,t+1}$
for all $e\in\delta v$.
Let $\pi_{t+1/2}$ be the
induced measure on the frozen spin
$x\in\set{\minus,\plus,\free}^2$.
Let $\pi$ be the measure on
the single spin $x^j\in\set{\minus,\plus,\free}$ that
is induced by $\starpi$.
Then update the frozen spin weights by setting, for both $j=1,2$,
	\beq\label{e:var.update.frozenupdate}
	(\lm_{t+1})^j(x^j)
	=(\lm_t)^j(x^j)
	\cdot
	\f{\pi(x^j) / \pi(\plus)}
	{(\pi_{t+1/2})^j(x^j) / (\pi_{t+1/2})^j(\plus)}\,.
	\eeq
This concludes the definition of the weights $\Lm_{t+1}$.\smallskip

\noindent\bemph{Part~2. Definition of error quantities.} Define the error quantities
	\[
	(\delta_t)^j
	\equiv
	\sum_{x\in\set{\minus,\free}}
	\Bigg|
	\f{\pi(x^j)/\pi(\plus)}
		{
		(\pi_t)^j(x^j)/(\pi_t)^j(\plus)
		}-1
	\Bigg|\,,\quad
	(\rho_{e,t})^j
	\equiv
	\bigg|
	\f{
	 \starpi_e(\red) / \starpi_e(\blu)}
	{ (\nu_{e,t})^j (\red) / (\nu_{e,t})^j(\blu)}
	-1
	\bigg|\,.
	\]
We will keep track of the aggregate errors
	\[
	\bm{\delta}_t
	\equiv\sum_{j=1,2} (\delta_t)^j\,,\quad
	\rho_{e,t}
	\equiv
	\sum_{j=1,2}
	(\rho_{e,t})^j\,,\quad
	\bm{\rho}_t
	\equiv \sum_{e\in\delta v}
		\rho_{e,t}\,.
	\]
For the initial measure $\nu_0=\nu_{\delta v}[\Lm_0;\hq]=\nu_{\delta v}[\Theta;\hq]$,
it follows by straightforward calculations that
	\begin{align*}
	(\delta_0)^j
	&\le O\Bigg(\sum_{e'\in\delta v}\bigg\{
		\ep_{e'}
		+\f{\dot{\ep}_{e'}}{2^k}
		+\f{\min\set{\ddot{\ep}_{e'},1}
		}{2^{k(1+\zeta)}}
		\bigg\}
		\Bigg) = O(\err_v)\,,\\
	(\rho_{e,0})^j
	&\le O\Bigg(
	\dot{\ep}_e
	+\f{\min\set{\ddot{\ep}_e,1}}{2^{k\zeta}}
	+\err_v
	\Bigg)
	\end{align*}
for both $j=1,2$ ---
this makes use of the assumptions~\eqref{e:condition.on.hat.messages} and \eqref{e:assume.variable.weights.close.to.one}.\smallskip

\noindent\bemph{Part~3. Effect of update on edge weights (from $t$ to $t+1/2$).} Recall that for each integer $t\ge0$, the procedure described in Part~1 goes from $t$ to $t+1/2$ by the update
\eqref{e:var.update.edgeupdate} on all the edge weights, leaving the frozen spin weights unchanged. For this update, Lemma~\ref{l:perturb.red.weights} gives
	\beq\label{e:redblue.update.on.updated.edge}
	(\rho_{e,t+1/2})^1
	\le
	\f{k^{O(1)}}{2^{k\zeta}}
	\Bigg\{(\rho_{e,t})^2
	+ \f{\bm{\rho}_t}{2^k}
	\Bigg\}\,,\eeq
with the analogous bound if we exchange $j=1,2$. 
Summing 
\eqref{e:redblue.update.on.updated.edge}
 over $j=1,2$ and over all $e\in\delta v$ gives
	\beq\label{e:redblue.update.on.updated.edge.simp}
	\bm{\rho}_{t+1/2}
	\le \f{k^{O(1)}}{2^{k\zeta}}
	\bm{\rho}_t\,.
	\eeq
Meanwhile, the error in the frozen spin marginals can change by at most
	\beq\label{e:redblue.update.on.frozen}
	\bm{\delta}_{t+1/2}
	\le \bm{\delta}_t 
	+\f{k^{O(1)} }{2^{k(1+\zeta)}}
	\bm{\rho}_t
	\,,\eeq
by a very similar calculation.\smallskip

\noindent\bemph{Part~4. Effect of update on frozen spin weights (from $t+1/2$ to $t+1$.} Next recall that for each integer $t\ge0$, the procedure of Part~1 goes from $t+1/2$ to $t+1$ by the update 
\eqref{e:var.update.frozenupdate}
on the frozen spin weights,
leaving the edge weights unchanged.
As a result the $\SPIN{red}$-to-$\SPIN{blue}$ ratios
are unaffected, so $(\rho_{e,t+1})^j=(\rho_{e,t})^j$.
As for the frozen spin marginals, the result of 
Lemma~\ref{l:transfer.error} gives
	\beq\label{e:frozen.update.frozen.result}
	\bm{\delta}_{t+1} 
	\le \f{k^{O(1)}\bm{\delta}_{t+1/2}}{2^{k\zeta}}
	\le \f{k^{O(1)}}{2^{k\zeta}}
	\Bigg\{
		\bm{\delta}_t
		+
		\f{\bm{\rho}_t}{2^{k(1+\zeta)}}
		\Bigg\}
	\eeq
where the last step is by \eqref{e:redblue.update.on.frozen}.\smallskip

\noindent\bemph{Part~5. Conclusion.}
For each integer $t\ge0$, when we go from $t$ to $t+1$ in the procedure of Part~1,
it follows by combining
\eqref{e:redblue.update.on.updated.edge.simp} and \eqref{e:frozen.update.frozen.result} that
	\[
	\begin{pmatrix}
	\bm{\delta}_{t+1}\\
	\bm{\rho}_{t+1}
	\end{pmatrix}
	\le
	\f{k^{O(1)}}{2^{k\zeta}}
	\begin{pmatrix}
	1 & 2^{-k(1+\zeta)} \\
	0 & 1
	\end{pmatrix}
	\begin{pmatrix}
	\bm{\delta}_t \\ \bm{\rho}_t
	\end{pmatrix}
	\]
(having used also that $\bm{\rho}_{t+1/2}=\bm{\rho}_t$). This implies
	\[
	\Upsilon_{t+1}
	\equiv
	\bm{\delta}_{t+1}
	+\f{\bm{\rho}_{t+1}}{2^k}
	\le
	\f{k^{O(1)}}{2^{k\zeta}}
	\Bigg\{
	\bm{\delta}_t 
	+ \f{\bm{\rho}_t}{2^k}
	\Bigg\}
	= \f{k^{O(1)}}{2^{k\zeta}} \Upsilon_t\,.
	\]
It follows that the iteration defined in Part~1 converges. The error on the frozen spin weights can be bounded as
	\begin{align*}
	\sum_{j=1,2}
	\sum_{x\in\set{\minus,\free}}
	\Bigg|\f{\lm^j(x)}{\theta^j(x)}-1\Bigg|
	&\le O\Bigg(
	\sum_{t=0}^\infty
	\sum_{j=1,2}
	\sum_{x\in\set{\minus,\free}}
	\bigg|\f{(\lm_{t+1})^j(x)}{(\lm_t)^j(x)}-1\bigg|
	\Bigg)\\
	&\le O\Bigg(
		\sum_{t=0}^\infty \bm{\delta}_t\Bigg)
	\le O\Bigg(
		\sum_{t=0}^\infty\Upsilon_t\Bigg)
	\le O(\Upsilon_0) \le O(\err_v)\,.
	\end{align*}
To bound the error on the edge weights, let us first note that for any $e\in\delta v$, 
\eqref{e:redblue.update.on.updated.edge} implies
	\[
	\sum_{t\ge1} \rho_{e,t}
	\le
	\f{k^{O(1)}}{2^{k\zeta}}
	\sum_{t\ge0}\Bigg\{
	\rho_{e,t} + \f{\bm{\rho}_t}{2^k}
	\Bigg\}
	\le 
	\f{k^{O(1)}}{2^{k\zeta}}
	\Bigg\{
	\sum_{t\ge1} \rho_{e,t}
	+\bigg(
	\rho_{e,0}
	+\sum_{t\ge0} \Upsilon_t\bigg)\Bigg\}\,,
	\]
and rearranging the inequality gives
	\[
	\sum_{t\ge1} \rho_{e,t}
	\le
	\f{k^{O(1)}}{2^{k\zeta}}
	\Bigg\{
	\rho_{e,0} + \sum_{t\ge0} \Upsilon_t\Bigg\}
	\le	\f{k^{O(1)}}{2^{k\zeta}}
	\Bigg\{
	\rho_{e,0} + \err_v\Bigg\}\,.
	\]
It follows from this that the total error on the edge weights can be bounded as
	\begin{align*}
	\sum_{j=1,2}
	\Bigg|\f{(\lm_e)^j}
		{(\theta_e)^j}-1\Bigg|
	&\le O\Bigg(\sum_{t=0}^\infty
	\rho_{e,t}\Bigg)
	\le
	 O\Bigg( \rho_{e,0}
	 	+ \f{k^{O(1)}}{2^{k\zeta}}
		\err_v
		\Bigg)
	\le
	O\Bigg(
	\dot{\ep}_e
	+\f{\min\set{\ddot{\ep}_e,1}}{2^{k\zeta}}
	+\err_v
	\Bigg)\,.
	\end{align*}
This concludes the proof.
\end{proof}

Next, recall
that for a variable-to-clause message $\dq_{va}$ 
on $\set{\RYGB}^2$,
we defined a reweighted version $Q_{va}$ by
\eqref{e:reweighted.messages.v.to.c}.
Now, for a clause-to-variable message $h\equiv h_{av}$ on $\set{\RYGB}^2$ we define the compensatory reweighting
	\beq\label{e:reweighted.messages.c.to.v}
	\rwhq(\sigma) =\f{ \tW(\sigma)}{Z_h}
		= \f{h(\sigma)
		(2^{|\pd a|-1})^{\red[\sigma]}}{Z_h}\,,
	\eeq
where $Z_h$ is the normalizing constant that makes $\rwhq$ a probability measure over $\set{\RYGB}^2$,
and we write $\tW$ for the non-normalized version of $\rwhq$. The following lemma records an elementary bound which will be used in the proof of Corollary~\ref{c:var.update} at the end of this subsection:

\begin{lem}\label{l:marginal.error}
Suppose at the edge $e=(av)$
we have variable-to-clause messages
$p$ and $q$
(both going from $v$ to $a$), and a clause-to-variable message $h$ (going from $a$ to $v$), all probability measures on $\set{\RYGB}^2$. Let $P,Q$ be the reweightings of $p,q$ defined by \eqref{e:reweighted.messages.v.to.c}. Let $\rwhq$ be the reweighting of $h$, with non-normalized version $\tW$, as defined by \eqref{e:reweighted.messages.c.to.v}. Assume
	\beq\label{e:beta.one.eighth}
	\bigg\|W -\f19
	\bigg\|_\infty \le \f1{18}\,.
	\eeq
Define the corresponding edge marginals
(probability measures on $\set{\RYGB}^2$)
	\[
	\pi_p(\sigma)
	\equiv
	\f{P(\sigma) W(\sigma)}
	{z_p}\,,\quad
	\pi_q(\sigma)
	\equiv
	\f{Q(\sigma) W(\sigma)}
	{z_q}\,,
	\]
where $z_p,z_q$ are the normalizing constants. If $\pi_p$ and $\pi_q$ have the same marginals on the first copy, $(\pi_p)^1=(\pi_q)^1$, then the first-copy marginals of $P$ and $Q$ satisfy
	\[\Bigg|\f{Q^1(\sigma^1)}{P^1(\sigma^1)}
		-1\Bigg|
	\le
	\overbrace{
	O(1)
	\bigg\|W -\f19\bigg\|_\infty
	\sum_\tau
	P(\tau)\Bigg(1 + 
	\f{\Ind{\tau^1=\sigma^1}}{P^1(\sigma^1)}
	\Bigg)
	\Bigg|\f{Q(\sigma)}{P(\sigma)}-1\Bigg|
	}^{\textup{this tighter bound is
	used in the proof of 
	Lemma~\ref{l:first.update.at.boundary}}}
	\le O(1)
	\bigg\|W -\f19\bigg\|_\infty
	\bigg\|\f{P}{Q} -1\bigg\|_\infty\,.
	\]
The analogous statement holds if we instead have 
$(\pi_p)^2=(\pi_q)^2$.

\begin{proof}
For the purposes of the proof, for $\sigma\in\set{\RYC}^2$ denote
	\[
	\delta(\sigma)
	\equiv \f{Q(\sigma)}{P(\sigma)}-1\,,\quad
	\beta(\sigma)
	\equiv W(\sigma)-\f19\,.
	\]
Note that, since $P$ is a probability measure over $\set{\RYC}^2$, we have
	\[
	z_p
	= \sum_\sigma P(\sigma) \bigg(W(\sigma)-\f19\bigg)
		+ \f19 \sum_\sigma P(\sigma)
	= \sum_\sigma P(\sigma) \beta(\sigma) +\f19
	\in \bigg[\f1{18},\f3{18}\bigg]\,,
	\]
where the last step uses the assumption 
\eqref{e:beta.one.eighth}. The same bound holds for $z_q$. Next we have
	\[
	\f{Q^1(\sigma^1)}{P^1(\sigma^1)}
	-1
	=
	\f{z_q-z_p}{z_p}
	+\f{z_q}{P^1(\sigma^1)}
	\bigg(
		\f{Q^1(\sigma^1)}{z_q}
	-\f{P^1(\sigma^1)}{z_p}\bigg)\,;
	\]
we will bound separately the two terms on the right-hand side. For the first term we have
	\beq\label{e:marginal.error.z.error}
	|z_q-z_p|
	=\Bigg| \sum_\sigma P(\sigma)
		\bigg(
		\f{Q(\sigma)}{P(\sigma)}-1
		\bigg)
		\bigg(
		\tW(\sigma)-\f19
		\bigg)\Bigg|
	=\Bigg| \sum_\sigma P(\sigma)
		\delta(\sigma)\beta(\sigma)\Bigg|
	\le
	\|\beta\|_\infty
	\sum_\sigma P(\sigma) |\delta(\sigma)|\,,
	\eeq
which implies $|z_q-z_p| \le \|\beta\|_\infty\|\delta\|_\infty$. For the second term,
direct algebraic manipulations give
	\begin{align*}
	\f{1}{9}\Bigg(
	\f{Q^1(\sigma^1)}{z_q}
	-\f{P^1(\sigma^1)}{z_p}\Bigg)
	&=
	\sum_{\sigma \in \set{\sigma^1}\times\set{\RYC}}
	\Bigg\{
	\f1{z_q}
	Q(\sigma)
	\bigg(W(\sigma) - \beta(\sigma)\bigg)
	-\f1{z_p}
	P(\sigma)
	\bigg(W(\sigma) - \beta(\sigma)\bigg)
	\Bigg\}\\
	&= 
	\sum_{\sigma \in \set{\sigma^1}\times\set{\RYC}}
	\beta(\sigma)
	\Bigg\{
	\f{P(\sigma)}{z_p}
	- \f{Q(\sigma)}{z_q}
	\Bigg\}
	+ \underbrace{
		 (\pi_q)^1(\sigma^1)- (\pi_q)^1(\sigma^1)
		 }_{\textup{zero}}\\
	&=\sum_{\sigma \in \set{\sigma^1}\times\set{\RYC}}
	\beta(\sigma)
	P(\sigma)\Bigg\{
	\f{z_q-z_p}{z_pz_q}
	-
	\f{1}{z_q}
	\bigg(
	\f{Q(\sigma)}{P(\sigma)}-1
	\bigg)
	\Bigg\}\,,
	\end{align*}
from which it follows (dividing through by $P^1(\sigma^1)$) that
	\[
	\f{1}{P^1(\sigma^1)}\Bigg|
	\f{Q^1(\sigma^1)}{z_q}
	-\f{P^1(\sigma^1)}{z_p}\Bigg|
	\le \|\beta\|_\infty \cdot O
	\Bigg( |z_q-z_p|
	+ 
	\sum_\tau
	\f{\Ind{\tau^1=\sigma^1}}{P^1(\sigma^1)}
	P(\tau)\delta(\tau)
	\Bigg)\,,
	\]
having used that $z_p,z_q \in[1/18,3/18]$.
Combining with \eqref{e:marginal.error.z.error} proves the lemma.
\end{proof}
\end{lem}

\begin{proof}[Proof of Corollary~\ref{c:var.update}]
Let $P,Q$ be the reweightings of $\dot{p},\dq$ defined by \eqref{e:reweighted.messages.v.to.c}.\smallskip

\noindent\bemph{Step 1. Non-marginal errors between $P$ and $Q$.} We claim that
Proposition~\ref{p:var.update} implies,
for each $e\in\delta v$,
	\beq\label{e:reweighted.var.msg.prelim.bd}
	\Bigg|\f{Q_e(\sigma)}
		{P_e(\sigma)}-1\Bigg|
	\le k^{O(1)}
		\Bigg( \err_v
	+ \Ind{\red[\sigma]\ge1}
	\bigg\{ \dot{\ep}_e
	+ \f{\min\set{\ddot{\ep}_e,1}}
	{2^{k\zeta}}
	\bigg\}\Bigg)
	\eeq
We will only briefly sketch the proof of
\eqref{e:reweighted.var.msg.prelim.bd}:
let
$g_e(\sigma)=\theta_e(\sigma)\hq_e(\sigma)$
and $h_e(\sigma)=\lm_e(\sigma)\hat{p}_e(\sigma)$,
and let $Z_{Q,e}$
and $Z_{P,e}$ be the normalizing constants such that
	\[
	\f{Z_{Q,e}Q_e(\red\red)}{Z_{P,e}P_e(\red\red)}
	=\Bigg\{
	\prod_{j=1,2}
	\f{\lm^j(\lit_e) \cdot (\lm_e)^j}
		{\theta^j(\lit_e) \cdot (\theta_e)^j}
	\Bigg\}
	\prod_{e'\in\delta v(\plus e)}
	\f{h_{e'}(\pur\pur)}{g_{e'}(\pur\pur)}
	\prod_{e'\in\delta v(\minus e)}
	\f{h_{e'}(\yel\yel)}{g_{e'}(\yel\yel)}\,.
	\]
By the assumptions of
Proposition~\ref{p:var.update} together with the resulting bounds on the error between $\theta_e$ and $\lm_e$, we have
	\[\Bigg|\f{h_e(\pur\pur)}{g_e(\pur\pur)}-1\Bigg|
	\le
	k^{O(1)}
	\Bigg(\ep_e 
	+ \f1{2^k}
	\bigg( \dot{\ep}_e+
	\f{\min\set{\ddot{\ep}_e,1}}
		{2^{k\zeta}}
	+\err_v\bigg) \Bigg)
	\]
for all $e\in\delta v$.
Substituting into the previous expression, and using 
Proposition~\ref{p:var.update} again, we conclude
	\begin{align*}
	\Bigg|
	\f{Z_{Q,e}Q_e(\red\red)}{Z_{P,e}P_e(\red\red)}-1
	\Bigg|
	&\le k^{O(1)}
	\Bigg(
	\err_v
	+\bigg\{ \dot{\ep}_e
	+\f{\min\set{\ddot{\ep}_e,1}}
		{2^{k\zeta}}
	+\err_v\bigg\}
	+\sum_{e'\in\delta v}
	\bigg\{ \ep_{e'}
	+ \f1{2^k}
	\bigg( \dot{\ep}_{e'}+
	\f{\min\set{\ddot{\ep}_{e'},1}}
	{(2^{k\zeta})_{e'}}
	+ \err_v\bigg) \bigg\}
	\Bigg) \\
	&\le k^{O(1)}
	\Bigg( \err_v
	+ \dot{\ep}_e
	+\f{\min\set{\ddot{\ep}_e,1}}
		{2^{k\zeta}}
	\Bigg)\,.
	\end{align*}
A similar calculation can be made for the other spins in $\set{\RYGB}^2$, and \eqref{e:reweighted.var.msg.prelim.bd} straightforwardly follows.\smallskip

\noindent\bemph{\hypertarget{c:var.update.BETTER.MGL.BOUND}{Step 2}. Marginal errors between $P$ and $Q$.} Fix an edge $e\in\delta v$. Let $\hat{s}_e=\hat{p}_e$, and $\hat{s}_{e'}=\hq_{e'}$ for $e'\in\delta v\setminus e$. We first apply Proposition~\ref{p:var.update} with $\hat{s}$ in place of $\hq$: let $\Gm$ be the resulting weights such that 
$\nu_{\delta v}[\Gm;\hat{s}]$ is judicious, and let $\dot{s}\equiv\BP[\Gm;\hat{s}]$ denote the corresponding outgoing messages from the variable, with reweighted versions $S$ as defined by \eqref{e:reweighted.messages.v.to.c}. Then we are exactly in the situation of Lemma~\ref{l:marginal.error}. The reweighting $W$ of $\hat{r}_e=\hat{p}_e$ satisfies
	\[
	\bigg\| W-\f19\bigg\|_\infty \le 
	\f{k^{O(1)}}{2^{k\zeta}}
	\]
by assumption~\eqref{e:condition.on.hat.messages}. We also note that $P(\sigma) \le O(1)/2^{k\red[\sigma]}$,
and likewise for $S(\sigma)$. It follows from
\eqref{e:reweighted.var.msg.prelim.bd} that
	\[
	\bigg\| \f{S}{P}-1\bigg\|_\infty
	\le k^{O(1)} \err_v\,.
	\]
Combining this with the result of Lemma~\ref{l:marginal.error} gives
the marginal error bound between $S$ and $P$,
	\beq\label{e:marginal.error.P.to.S}
	\bigg\| \f{S^j}{P^j}-1\bigg\|_\infty
	\le 
	O(1)
	\bigg\| W-\f19\bigg\|_\infty
	\bigg\| \f{S}{P}-1\bigg\|_\infty
	\le
	\f{k^{O(1)} \err_v}
	{2^{k\zeta}}
	\eeq
for both $j=1,2$. It also follows from 
\eqref{e:reweighted.var.msg.prelim.bd} that
	\beq\label{e:S.Q.error}
	\bigg|\f{Q(\sigma)}{S(\sigma)}-1\bigg|
	\le k^{O(1)} \Bigg( 
	\ep_e 
	+\f1{(2^k)^{\Ind{\red[\sigma]=0}}}
	\bigg\{ \dot{\ep}_e
	+ \f{\min\set{\ddot{\ep}_e,1}}{2^{k\zeta}}
	\bigg\}\Bigg)\,,
	\eeq
from which we obtain the marginal error bound
	\beq\label{e:marginal.error.S.to.Q}
	\bigg|\f{Q^j(\tau)}{S^j(\tau)}-1\bigg|
	\le k^{O(1)}\Bigg( 
	\ep_e 
	+\f1{(2^k)^{\Ind{\tau\ne\red}}}
	\bigg\{ \dot{\ep}_e
	+ \f{\min\set{\ddot{\ep}_e,1}}{2^{k\zeta}}
	\bigg\}\Bigg)\,.
	\eeq
The result follows by combining
\eqref{e:reweighted.var.msg.prelim.bd},
\eqref{e:marginal.error.P.to.S}, and \eqref{e:marginal.error.S.to.Q}.
\end{proof}

\subsection{Contraction in non-defective trees}
\label{ss:contraction.nondefect}

In this subsection we complete the proof of Proposition~\ref{p:nice.tree.lagrange}.
As in the statement of the proposition, let $U$ be a compound enclosure, and take
a subtree $T\subseteq U$ of the form described in Definition~\ref{d:nu.judicious}, such that $T$ contains no defective variables. Recall from \eqref{e:clause.xi.prelim.defn} that for any clause $a$ in $T$, and for any given set of boundary marginals $\omega_{\delta T}$, we defined the quantity
	\beq\label{e:clause.xi}
	\xi_a(T;\omega_{\delta T})
	\equiv
	\sum_{e\in\delta T}
	\bigg(\f{(\vth)^{1/4}}{2^k}\bigg)^{\BTW_T(e,a)}
	\disc_e(\omega)\,,
	\eeq
for $\disc_e(\omega)$ as given by \eqref{e:def.discrepancy.measure.e}.
The main technical result of this subsection is the following:

\begin{ppn}\label{p:induct}
In the setting of Proposition~\ref{p:nice.tree.lagrange},
for the iterative construction of Definition~\ref{d:nice.lagrange.defn.weights}, 
on any edge $(av)$ in $T$ we have
(using the notations from 
Definitions~\ref{d:clause.rel.error}~and~\ref{d:var.rel.error})
	\begin{alignat}{2}
	\label{e:induction.variable.error}
	\vrelerr
	( \dq_{va,t-1},\dq_{va,t} )
	\equiv
	\bm{\delta}_{va}(t)
	\equiv
	\begin{pmatrix}
	\delta_{va,t} \\
	\dot{\delta}_{va,t} \\
	\ddot{\delta}_{va,t} \\
	\mdel_{va,t}\\
	\mdelred_{va,t}
	\end{pmatrix}
	&\le
	k^{O(1)} \xi_a
	(k^{O(1)}(\vth)^{1/3})^{t-1}
	\vec{\delta}\,, &\quad
	\vec{\delta}
	\equiv
		\begin{pmatrix}
		\delta \\ 
		\dot{\delta} \\
		\ddot{\delta} \\
		\mdel \\ 
		\mdelred
		\end{pmatrix}
	&\equiv
		\begin{pmatrix}
		1 \\ 
		(\vth)^{-1} \\
		(\vth)^{-1} \\
		(\vth)^2 \\ 
		\vth
		\end{pmatrix}\,, \\
	\crelerr(
	\hq_{av,t-1},
	\hq_{av,t}
	)
	\equiv
	\bm{\ep}_{av}(t)
	\equiv
	\begin{pmatrix}
	\ep_{av,t}\\
	\dot{\ep}_{av,t}\\
	\ddot{\ep}_{av,t}
	\end{pmatrix}
	&\le k^{O(1)} \xi_a (k^{O(1)}(\vth)^{1/3})^{t+1} 
	\vec{\ep}\,, &\quad
	\vec{\ep}
	\equiv
		\begin{pmatrix}
		\ep \\ 
		\dot{\ep} \\
		\ddot{\ep} \end{pmatrix}
	&\equiv 
		\begin{pmatrix}
		2^{-k} \\ 
		\vth \\
		(\vth)^{-1} \end{pmatrix}\,,
	\label{e:induction.clause.error}
	\end{alignat}
where $\xi_a\equiv\xi_a(T;\omega_{\delta T})$
is defined by \eqref{e:clause.xi}.
\end{ppn}

We first provide some lemmas towards the proof of Proposition~\ref{p:induct}. The following is an elementary bound:

\begin{lem}\label{l:sum.around.clause}
In the setting of Propositions~\ref{p:nice.tree.lagrange}~and~\ref{p:induct}, if $N(a)$ denotes the clauses in $T$ at unit distance from a clause $a$ in $T$, then
we have the bound
	\[
	\sum_{b\in N(a)}
	\xi_b(T;\omega_{\delta T})
	\le O\Bigg( \f{k^{O(1)} 2^k
		\xi_a(T;\omega_{\delta T})}{(\vth)^{1/4}}
	\Bigg)\,.
	\]

\begin{proof}
All the variables in $T$ are assumed to be non-defective, hence nice. Let $d_{\max} = O(k 2^k)$ denote the maximum degree of a nice variable
(Definition~\ref{d:nice}). Then, for any edge $e\in\delta T$ and any clause $a$ in $T$, we have
	\begin{align*}
	\bigg|\Big\{
		b\in N(a) : \BTW_T(e,b) = \BTW_T(e,a)-1
		\Big\}\bigg|
		&\le1\,,\\
	\bigg|\Big\{
	b\in N(a) : \BTW_T(e,b) = \BTW_T(e,a)
	\Big\}\bigg|
		&\le d_{\max}\,,\\
	\bigg|\Big\{
	b\in N(a) : \BTW_T(e,b) = \BTW_T(e,a)+1
		\Big\}\bigg|
		&\le k d_{\max}\,.
	\end{align*}
Combining these bounds with the definition \eqref{e:clause.xi} gives
	\[
	\sum_{b\in N(a)}
	\xi_b(T;\omega_{\delta T})
	\le
	\sum_{e\in\delta T}
	\disc_e(\omega)
	\bigg(\f{(\vth)^{1/4}}{2^k}\bigg)^{\BTW_T(e,a)}
	\bigg\{
	\f{2^k}{(\vth)^{1/4}}
	+ d_{\max}
	+ kd_{\max}
	\cdot \f{(\vth)^{1/4}}{2^k}
	\bigg\}
	\le
	O\Bigg(
	\f{2^k \xi_a(T;\omega_{\delta T})}{(\vth)^{1/4}}
	\Bigg)\,,
	\]
as claimed.
\end{proof}
\end{lem}

Next, we give a bound on the errors after the first round of boundary updates:

\begin{lem}\label{l:first.update.at.boundary}
Recall the notations of Definitions~\ref{d:clause.rel.error}~and~\ref{d:var.rel.error}.
In the setting of
Propositions~\ref{p:nice.tree.lagrange}~and~\ref{p:induct}, 
for any boundary edge $e=(av)\in\delta T$,
suppose that 
we have,
for an absolute constant $0<\zeta\le1/20$, the bounds
	\beq\label{e:assume.omega.error.gamma}
	\crelerr(\prodom_e,\omega_e)
	\le \begin{pmatrix} \gamma_e \\ \dot{\gamma}_e\\
		\ddot{\gamma}_e
		\end{pmatrix}
	\le
	k^{O(1)}
	\min\left\{
	\disc_e(\omega)
	\begin{pmatrix}
	1 \\ (\vth)^{-1} \\ (\vth)^{-1}
	\end{pmatrix}
	,
	 \f{1}{2^{2k\zeta}}
	\begin{pmatrix} 1 \\ 1 \\ 2^k \end{pmatrix}
	\right\}
	\,,
	\eeq
where the last bound follows from the definition
\eqref{e:def.discrepancy.measure.e}
together with the assumption
\eqref{e:apriori.assump.omega}
(taking the entrywise minimum of the two vectors). Then, after the first round of boundary updates
(i.e., after applying
Definition~\ref{d:nice.lagrange.defn.weights}
Step~I), we have
	\[
	\vrelerr(\dq_{e,0},\dq_{e,1})
	\le k^{O(1)} \begin{pmatrix}
	1 & 2^{-k} & 4^{-k} \\
	1 & 1 & 4^{-k} \\
	1 & 2^{-k} & 1 \\
	2^{-k/10} & 2^{-k/10} 2^{-k} & 2^{-k/10}4^{-k} \\
	2^{-k/10} & 2^{-k/10} & 2^{-k/10} 2^{-k}
	\end{pmatrix}
	\begin{pmatrix} \gamma_e \\ \dot{\gamma}_e\\
		\ddot{\gamma}_e
		\end{pmatrix}
	\le k^{O(1)} \min\left\{
	\disc_e(\omega)
	\begin{pmatrix}
	1 \\ (\vth)^{-1} \\ (\vth)^{-1}\\
	2^{-k/10}(\vth)^{-1}\\
	2^{-k/10}(\vth)^{-1}
	\end{pmatrix},
	\f1{2^{k\zeta}}
	\begin{pmatrix}
	1\\1\\ 2^k \\ 1\\ 1
	\end{pmatrix}
	\right\}\,.
	\]
Moreover, if $Q_{e,1}$ is the reweighted version of
$\dq_{e,1}$ defined by
 \eqref{e:reweighted.messages.v.to.c},
 then it satisfies the bounds 
\eqref{e:first.update.at.boundary}.

\begin{proof}
Recall that $\dq_{e,0}=\proddq_e=\dqbul_e\otimes\dqbul_e$.
From Definition~\ref{d:nice.lagrange.defn.weights},
the first updates (Step~I) yield boundary messages
	\[
	\dq_{e,1}(\sigma)
	\cong \f{\omega_e(\sigma)}{\prodhq_e(\sigma)}
	= \f{\omega_e(\sigma)}
		{\hqbul_e(\sigma^1)
		\hqbul_e(\sigma^2)}
	\,.
	\]
We use \eqref{e:reweighted.messages.v.to.c} to define the reweighted versions
	\[
	Q_{e,1}(\sigma)
	\cong
	\f1{(2^{|\pd a|-1})^{\red[\sigma]}}
	\f{\omega_e(\sigma)}
		{\prodhq_e(\sigma)}
	\cong
	\f{\omega_e(\sigma)}{\prodom_e(\sigma)}
	\f{\proddq_e(\sigma)}
		{(2^{|\pd a|-1})^{\red[\sigma]}}
	\cong
	\f{\omega_e(\sigma)}{\prodom_e(\sigma)}
	\prodQ_e(\sigma)\,,
	\]
where the second-to-last step uses that 
$\prodom_e(\sigma)\cong\proddq_e(\sigma)\prodhq_e(\sigma)$, and $\prodQ_e = Q_{e,0}$ is the reweighting of $\proddq$ by \eqref{e:reweighted.messages.v.to.c}. Therefore, there exists a normalizing constant $Z_e$ such that
	\[
	Z_e Q_{e,1}(\sigma)
	= \f{\omega_e(\sigma)}{\prodom_e(\sigma)}
	\prodQ_e(\sigma)\,.
	\]
Summing the equation over $\sigma\in\set{\RYGB}^2$, and recalling the assumption
\eqref{e:assume.omega.error.gamma}, we find
	\begin{align*}
	Z_e
	&= \sum_{\sigma:\red[\sigma]=0}
	\prodQ_e(\sigma)
	\Big(1 + O(\gamma_e)\Big)
	+ \sum_{\sigma:\red[\sigma]=1}
	\prodQ_e(\sigma)
	\Big(1 + O(\dot{\gamma}_e)\Big)
	+\prodQ_e(\red\red)
	\Big(1 + O(\ddot{\gamma}_e)\Big)\\
	&=
	1 + O\bigg(
	\gamma_e + \f{\dot{\gamma}_e}{2^k}
	+ \f{\ddot{\gamma}_e}{4^k}\bigg)\,.
	\end{align*}
Recalling that $Q_{e,0}=\prodQ_e$, it follows by combining the above estimates that
	\[\Bigg|
	\f{Q_{e,1}(\sigma)}{Q_{e,0}(\sigma)}
	-1\Bigg|
	\le 
	k^{O(1)}
	\Bigg(
	\gamma_e + \f{\dot{\gamma}_e}{2^k}
	+ \f{\ddot{\gamma}_e}{4^k}
	+ \Ind{\red[\sigma]=1} \dot{\gamma_e}
	+ \Ind{\red[\sigma]=2} \ddot{\gamma_e}
	\Bigg)\,.
	\]
Next recall that from the initialization of Definition~\ref{d:nice.lagrange.defn.weights} we have $\prodom_e\cong\prodhq_e\proddq_e= \hq_{e,0}\dq_{e,0}$; and
after the first round of boundary updates we have
$\omega_e\cong
\hq_{e,0}\dq_{e,1}$. Both $\prodom_e$ and $\omega_e$ have single-copy marginals $\starpi_e$, so we are precisely in 
the situation of Lemma~\ref{l:marginal.error}
with $h=\prodhq_e$. Applying Lemma~\ref{l:marginal.error} gives, for both $j=1,2$,
	\[
	\Bigg|
	\f{(Q_{e,1})^j(\sigma)}
	{(Q_{e,0})^j(\sigma)}
	-1\Bigg|
	\le
	\f{k^{O(1)}}{2^{k/10}}
	\begin{cases}
	\gamma_e + \dot{\gamma}_e/2^k
		+\ddot{\gamma}_e/4^k
	&\textup{for }\sigma\ne\red\,,\\
	\gamma_e+\dot{\gamma}_e + \ddot{\gamma}_e/2^k
	&\textup{for }\sigma=\red\,.\\
	\end{cases}
	\]
This proves the first claimed bound on
$\vrelerr(\dq_{e,0},\dq_{e,1})$.
The remaining bounds follow using 
\eqref{e:assume.omega.error.gamma}.
\end{proof}
\end{lem}

\begin{cor}\label{c:first.clause.update.round}
In the setting of Propositions~\ref{p:nice.tree.lagrange}~and~\ref{p:induct},
after the first update of the clauses just above the boundary variables
(i.e., after applying
Definition~\ref{d:nice.lagrange.defn.weights}
Step~I at each boundary edge,
followed by Step II(a) at each clause incident to a boundary edge), we have
	\[
	\crelerr( \hq_{e',1},\hq_{e',0})
	\le k^{O(1)}
	\begin{pmatrix} 
		2^{-k(1+\zeta)}\\ 
		2^{-k\zeta}\\ 
		2^{-k\zeta}
		\end{pmatrix}
	\]
for every edge $e'\in\delta a \setminus\delta T$ for any clause $a$ incident to $\delta T$.

\begin{proof} 
For all $e\in\delta T$,
the error
$\vrelerr(\dq_{e,1},\dq_{e,0})$
is bounded by Lemma~\ref{l:first.update.at.boundary}.
Applying Proposition~\ref{p:clause.update} gives, after applying Step II(a) at the clauses above the boundary variables,
	\[
	\crelerr( \hq_{e',1},\hq_{e',0})
	\le
	k^{O(1)}
	\begin{pmatrix}
	2^{-k(1+\zeta)}
	& 4^{-k} & 2^{-k(1+\zeta)} & 2^{-k} & 2^{-k}\\
	2^{-k} & 2^{-k} & 2^{-k(1+\zeta)}
		& 1 & 2^{-k} \\
	1 & 4^{-k} & 2^{-k(1+\zeta)}
		& 2^{-k} & 2^{-k}
	\end{pmatrix}
	\begin{pmatrix}
	2^{-k\zeta}\\
	2^{-k\zeta}\\
	\min\set{2^{k(1-\zeta)},1} \\
	2^{-k\zeta}\\
	2^{-k\zeta}
	\end{pmatrix}
	\le k^{O(1)}
	\begin{pmatrix}
	2^{-k(1+\zeta)} \\
	2^{-k\zeta}\\2^{-k\zeta}
	\end{pmatrix}
	\]
for the edges $e'$ just above $\delta T$.
This proves the claim.
\end{proof}
\end{cor}

\begin{proof}[\hypertarget{p:induct.proof}{Proof of Proposition~\ref{p:induct}}]
We will prove the result by induction. We divide the proof into a few steps, indicated by Arabic numerals below. We use Roman numerals to refer to the steps of Definition~\ref{d:nice.lagrange.defn.weights}.\smallskip

\noindent\bemph{Step 1. Verification of base case.}
By Lemma~\ref{l:first.update.at.boundary},
for any boundary edge $e=(av)\in\delta T$, we have
	\beq\label{e:initialization.rel.to.eigenvec}
	\bm{\delta}_{va}(1)
	\equiv\begin{pmatrix}
	\delta_{va,1}\\
	\dot{\delta}_{va,1}\\
	\ddot{\delta}_{va,1}\\
	\mdel_{va,1}\\
	\mdelred_{va,1}\\
	\end{pmatrix}
	\le k^{O(1)}
	\disc_e(\omega)
	\begin{pmatrix}
	1 \\ (\vth)^{-1} \\ (\vth)^{-1}\\
	2^{-k/10}(\vth)^{-1}\\
	2^{-k/10}(\vth)^{-1}
	\end{pmatrix}
	\le k^{O(1)} \xi_a
	\begin{pmatrix}
	1 \\ (\vth)^{-1} \\ (\vth)^{-1}\\
	(\vth)^2\\ \vth
	\end{pmatrix}
	= k^{O(1)} \xi_a \vec{\delta}\,,
	\eeq
which verifies the bound 
\eqref{e:induction.variable.error} for $t=1$.\smallskip

\noindent\bemph{Step 2. Bounds for clause updates.}
For any clause $a\ne\crt$ in the tree $T$, for both steps II(a) and III(b), our earlier result Proposition~\ref{p:clause.update} implies the bound
	\[\bm{\ep}_{av}(t)\equiv
	\begin{pmatrix}
	\ep_{av,t}\\
	\dot{\ep}_{av,t}\\
	\ddot{\ep}_{av,t}
	\end{pmatrix}
	\le k^{O(1)}
	\begin{pmatrix}
	2^{-k(1+\zeta)}
	& 4^{-k} & 2^{-k(1+\zeta)} & 2^{-k} & 2^{-k}\\
	2^{-k} & 2^{-k} & 2^{-k(1+\zeta)}
		& 1 & 2^{-k} \\
	1 & 4^{-k} & 2^{-k(1+\zeta)}
		& 2^{-k} & 2^{-k}
	\end{pmatrix}
	\sum_{u\in\pd a \setminus v}
	\begin{pmatrix}
		\delta_{ua,t-1}\\
		\dot{\delta}_{ua,t-1}\\
		\min\set{\ddot{\delta}_{ua,t-1},1}\\
		\mdel_{ua,t-1}\\
		\mdelred_{ua,t-1}
	\end{pmatrix}\,.\]
(On the right-hand side of 
the above, 
for Step~III(b) one can omit the contribution from the parent variable of $a$; but we will not use this fact.) Substituting the inductive hypothesis
\eqref{e:induction.variable.error} gives the simplified bound
	\beq\label{e:simplified.clause.update.bound}
	\bm{\ep}_{av}(t)
	\le k^{O(1)}
	\xi_a (k^{O(1)}(\vth)^{1/3})^{t-2}
	\underbrace{
	\begin{pmatrix}
	2^{-k(1+\zeta)}
	& 4^{-k} & 2^{-k(1+\zeta)} & 2^{-k} & 2^{-k}\\
	2^{-k} & 2^{-k} & 2^{-k(1+\zeta)}
		& 1 & 2^{-k} \\
	1 & 4^{-k} & 2^{-k(1+\zeta)}
		& 2^{-k} & 2^{-k}
	\end{pmatrix}
	}_{\textup{denote this }\widehat{\MAT}}
	\vec{\delta}\,.\eeq
At the root clause $\crt$ we make the trivial update
\eqref{e:message.from.top.always.product}, so
$\crelerr(\hq_{\crt\vrt,t-1},\hq_{\crt\vrt,t})
=(0,0,0)\in\mathbb{R}^3$.\smallskip

\noindent\bemph{Step 3. Bounds for internal variable updates.}
For any internal variable $v$ in $T$, for both steps II(b) and III(a), applying Corollary~\ref{c:var.update} gives the bound 
	\begin{align*}
	\bm{\delta}_{va}(t)
	\equiv
	\begin{pmatrix}
	\delta_{va,t} \\
	\dot{\delta}_{va,t} \\
	\ddot{\delta}_{va,t} \\
	\mdel_{va,t}\\
	\mdelred_{va,t}
	\end{pmatrix}
	&\le k^{O(1)}
	\sum_{c\in\pd v\setminus a}
	\begin{pmatrix}
	1\\1\\1\\ 2^{-k\zeta} \\ 2^{-k\zeta}
	\end{pmatrix} 
	\begin{pmatrix}
	1 & 2^{-k} & 2^{-k(1+\zeta)}
	\end{pmatrix}
	\begin{pmatrix}
	\ep_{cv,t-1}\\
	\dot{\ep}_{cv,t-1}\\
	\min\set{\ddot{\ep}_{cv,t-1},1}
	\end{pmatrix}\\
	&\qquad+k^{O(1)}
	\begin{pmatrix}
	0&0&0\\
	0&1&2^{-k\zeta}\\
	0&1&2^{-k\zeta}\\
	1&2^{-k}&2^{-k(1+\zeta)}\\
	1&1&2^{-k\zeta}\\
	\end{pmatrix}
	\begin{pmatrix}
	\ep_{av,t-1}\\
	\dot{\ep}_{av,t-1}\\
	\min\set{\ddot{\ep}_{av,t-1},1} 
	\end{pmatrix}
	\,.\end{align*}
(On the right-hand side of the above,
for Step~II(b) one only needs the first term that sums over $c\in\pd v\setminus a$. For Step~III(a) one can omit the contribution to the sum from the parent clause of $v$. We will not use this fact.) We will bound the above sum over $c\in\pd v\setminus a$ by the sum over all $c\in N(a)$. Then, substituting the inductive hypothesis
\eqref{e:induction.clause.error}
and combining with Lemma~\ref{l:sum.around.clause} gives the simplified bound
	\begin{align}\nonumber
	\bm{\delta}_{va}(t)
	&\le k^{O(1)}\xi_a
	(k^{O(1)}(\vth)^{1/3})^t
	\left\{
	\f{2^k}{(\vth)^{1/4}}
	\begin{pmatrix}
	1\\1\\1\\ 2^{-k\zeta} \\2^{-k\zeta}
	\end{pmatrix} 
	\begin{pmatrix}
	1 & 2^{-k} & 2^{-k(1+\zeta)}
	\end{pmatrix}
	+\begin{pmatrix}
	0&0&0\\
	0&1&2^{-k\zeta}\\
	0&1&2^{-k\zeta}\\
	1&2^{-k}&2^{-k(1+\zeta)}\\
	1&1&2^{-k\zeta}\\
	\end{pmatrix}
	\right\} \vec{\ep}\\
	&\le k^{O(1)} \xi_a
	(k^{O(1)}(\vth)^{1/3})^t
	\underbrace{
	\f1{(\vth)^{1/4}}
	\begin{pmatrix}
	2^k & 1 & 2^{-k\zeta} \\
	2^k & 1 & 2^{-k\zeta} \\
	2^k & 1 & 2^{-k\zeta} \\
	2^{k(1-\zeta)}&2^{-k\zeta}&2^{-2k\zeta} \\
	2^{k(1-\zeta)}& (\vth)^{1/4} & (\vth)^{1/4} 2^{-k\zeta}
	\end{pmatrix}
	}_{\textup{denote this $\dot{M}_\textup{int}$}}
	\vec{\ep}\,.
	\label{e:simplified.var.update.bound}
	\end{align}\smallskip

\noindent\bemph{Step 4. Bounds for boundary updates.} For each boundary variable $v$ in $T$ with parent clause $a$, in Step~I we have
	\beq\label{e:simplified.boundary.update.bound}
	\bm{\delta}_{va}(t)
	\le k^{O(1)}
	\begin{pmatrix}
	1 & 2^{-k} & 2^{-k(1+\zeta)}\\
	1 & 1 & 2^{-k(1+\zeta)}\\
	1 & 2^{-k} & 1\\
	1 & 2^{-k} & 2^{-k(1+\zeta)}\\
	1 & 1 & 2^{-k\zeta}
	\end{pmatrix}
	\begin{pmatrix}
	\ep_{av,t-1}\\
	\dot{\ep}_{av,t-1}\\
	\ddot{\ep}_{av,t-1}
	\end{pmatrix}
	\le k^{O(1)}\xi_a
	(k^{O(1)}(\vth)^{1/3})^t
	\underbrace{
	\begin{pmatrix}
	1 & 2^{-k} & 2^{-k(1+\zeta)}\\
	1 & 1 & 2^{-k(1+\zeta)}\\
	1 & 2^{-k} & 1\\
	1 & 2^{-k} & 2^{-k(1+\zeta)}\\
	1 & 1 & 2^{-k\zeta}
	\end{pmatrix}
	}_{\textup{denote this }\dot{\MAT}_\textup{bd}}
	\vec{\ep}\,,
	\eeq
where the last bound again uses the inductive hypothesis \eqref{e:induction.clause.error}.\smallskip

\noindent\bemph{Step 5. Verification of induction.}
For the matrices defined in \eqref{e:simplified.clause.update.bound},
\eqref{e:simplified.var.update.bound}, and \eqref{e:simplified.boundary.update.bound}, we have
	\beq\label{e:approx.eigenvector.bounds}
	\widehat{\MAT}
	\vec{\delta}
	\le k^{O(1)} \vth\vec{\ep}\,,\quad
	\Big(\dot{\MAT}_\textup{int}
	+\dot{\MAT}_\textup{bd}\Big)
	\vec{\ep}
	\le
	\f{k^{O(1)}}{(\vth)^{1/4}}
	\begin{pmatrix}
	1\\1\\1\\k^c/2^{k\zeta}\\(\vth)^{5/4}
	\end{pmatrix}
	+ k^{O(1)}
	\begin{pmatrix}
	2^{-k}\\
	\vth\\
	(\vth)^{-1}\\
	2^{-k}\\
	\vth
	\end{pmatrix}
	\le \f{k^{O(1)}}{(\vth)^{1/4}}
	\vec{\delta}\,.
	\eeq
Using these inequalities, the bounds
\eqref{e:simplified.clause.update.bound},
\eqref{e:simplified.var.update.bound}, and \eqref{e:simplified.boundary.update.bound}
simplify to 
	\begin{align*}
	\bm{\ep}_{av}(t)
	&\le k^{O(1)}\xi_a
	(k^{O(1)}(\vth)^{1/3})^{t-2}
	\vth\vec{\ep}
	\le k^{O(1)}\xi_a
	(k^{O(1)}(\vth)^{1/3})^{t+1}
	\vec{\ep}
	\,, \\
	\bm{\delta}_{va}(t)
	&\le k^{O(1)}\xi_a 
	(k^{O(1)}(\vth)^{1/3})^t
	(\vth)^{-1/4} \vec{\delta}
	\le k^{O(1)}\xi_a 
	(k^{O(1)}(\vth)^{1/3})^{t-1}
	\vec{\delta}\,.
	\end{align*}
This verifies the inductive hypotheses
\eqref{e:induction.variable.error}~and~\eqref{e:induction.clause.error}.\smallskip

\noindent\bemph{Step 6. Conclusion.} We conclude the proof by briefly addressing some points which we neglected in the above. In Step~2, in order to apply Proposition~\ref{p:clause.update}
we needed to check that the
messages incoming to the clauses
satisfy the estimates
\eqref{e:first.update.at.boundary}.
This can be seen by induction, with the base case given directly by Lemma~\ref{l:first.update.at.boundary}.
Similarly, in Step~3, in order to apply Corollary~\ref{c:var.update}
we needed to check that the messages incoming to the internal variables satisfy the estimates
\eqref{e:condition.on.hat.messages}. This can also be seen by induction, with the base case given directly by Corollary~\ref{c:first.clause.update.round}. This finishes the proof of the proposition.
\end{proof}

\begin{rmk}[explanation of choices in \eqref{e:induction.variable.error}~and~\eqref{e:induction.clause.error}]
In this remark we give some explanation for the choice of the vectors
$\vec{\delta}$ and $\vec{\ep}$
in Proposition~\ref{p:induct}.
If we take the product of the $3\times 5$ matrix in \eqref{e:simplified.clause.update.bound}
with the $5\times3$ matrix in
\eqref{e:simplified.var.update.bound}, the result can be (entrywise) upper bounded as
	\[
	\hat{\MAT}
	\dot{\MAT}_\textup{int}
	\le
	k^{O(1)}
	\begin{pmatrix}
	2^{-k\zeta} & (\vth)^{1/4} 2^{-k} &
		(\vth)^{1/4} 2^{-k(1+\zeta)}\\
	2^{k(1-\zeta)} & 2^{-k\zeta} & 2^{-k\zeta}\\
	2^k & 1 & 2^{-k\zeta}
	\end{pmatrix}
	\le k^{O(1)}
	\begin{pmatrix}
	2^{-k} \\
	\vth \\
	(\vth)^{-1}
	\end{pmatrix}
	\begin{pmatrix}
	2^k\vth & (\vth)^{1/4} & (\vth)^5
	\end{pmatrix}\,.
	\]
For the last bound on the right-hand side,
the vector $\vec{\ep}$
from \eqref{e:induction.clause.error} from Proposition~\ref{p:induct} is a right eigenvector with eigenvalue $O(\vth)$. This explains why
$\vec{\ep}$ is a good choice for our purposes (although it is certainly not the unique choice that would give a sufficiently good bound). On the other hand, given $\vec{\ep}$, we chose the vector $\vec{\delta}$ to satisfy the bounds
\eqref{e:approx.eigenvector.bounds}
and
\eqref{e:initialization.rel.to.eigenvec}.
(Again, it is certainly not the unique one that suffices for our purposes.)
\end{rmk}

\begin{proof}[\hypertarget{p:nice.tree.lagrange.proof}{Proof of Proposition~\ref{p:nice.tree.lagrange}}]
It is clear from
Proposition~\ref{p:induct}
that the construction of
Definition~\ref{d:nice.lagrange.defn.weights} converges to the desired Lagrangian weights
$\Lm\equiv\Lm(T;\omega_{\delta T})$.
The corresponding Gibbs measure $\nu=\nu[T;\Lm]$
has then edge marginals $\nu_{av} \cong \dq_{va} \hq_{av}$ where $\dq_{va}\equiv\dq_{va,\infty}$
and $\hq_{av}\equiv\hq_{av,\infty}$
are the limiting \textsc{bp} messages. 
It remains only to verify the estimate
\eqref{e:clause.xi.prelim.defn}. By Proposition~\ref{p:induct}, on any edge $(av)$ we have, abbreviating
$\xi_a\equiv\xi_a(T;\omega_{\delta T})$,
	\begin{align*}
	\vrelerr(\proddq_{va},\dq_{va})
	\le
	\sum_{t\ge1}
	\vrelerr(\dq_{va,t-1},\dq_{va,t})
	&\le O(1)
	\xi_a \vec{\delta}\,,\\
	\crelerr(\prodhq_{av},\hq_{av})
	\le
	\sum_{t\ge1}
	\crelerr(\hq_{av,t-1},\hq_{av,t})
	&\le O(1)
	\xi_a (k^{2c}(\vth)^{1/3})^2 \vec{\ep}\,.
	\end{align*}
This implies, for 
all edges $(av)$ in $T$ and all $\sigma\in\set{\RYGB}^2$, the bound
	\[\Bigg|\f{\dq_{va}(\sigma)}{\proddq_{va}(\sigma)}
	-1\Bigg| 
	+\Bigg|\f{\hq_{av}(\sigma)}{\prodhq_{av}(\sigma)}
	-1\Bigg| 
	\le O\Bigg(
	\f{\xi_a}{(\vth)^{\Ind{\red\in\set{\sigma^1,\sigma^2}}}}
	\Bigg)\,.
	\]
Since $\nu_{av} \cong \dq_{va} \hq_{av}$, it follows straightforwardly that
	\[\Bigg|
	\f{\nu_{av}(\sigma)}
	{\prodom_{av}(\sigma)}-1\Bigg|
	\le
	O\Bigg(
	\f{\xi_a}{(\vth)^{\Ind{\red\in\set{\sigma^1,\sigma^2}}}}
	\Bigg)\,.
	\]
Recalling
\eqref{e:def.discrepancy.measure.e}
gives
$\disc_{av}(\nu_{av})\le O(\xi_a)$, as claimed.
\end{proof}

\subsection{Contraction with multiple clause types}
\label{ss:ctypes.contract}

In this subsection we prove Proposition~\ref{p:contraction.for.simple.var}, which concerns the entropy maximization problem from Proposition~\ref{p:block.update.non.compound}.
The following is an analogue of
Definition~\ref{d:param.weights.Lambda}:

\begin{dfn}[parametrization of Lagrangian weights $\Psi$
on a depth-one neighborhood of a non-compound variable]
\label{d:param.weights.Augmented}
Let $U$ be the depth-one tree from Definition~\ref{d:judicious.augmented.alphabet}, rooted at a variable $v$ of (non-compound) total type $\bT$. For an augmented pair coloring $(\usi,\uL)$ of $U$, we will define weight functions $\Psi\equiv\Psi_U$ parametrized as
	\[
	\Psi(\usi,\uL)
	=
	\Psi_v(\usi_{\delta v},\uL_{\delta v})
	\cdot
	\Bigg\{
	\prod_{e\in\delta U}
	\psi_e(\sigma_e,\bL_e)
	\Bigg\}\,,\]
where the root variable weight $\Psi_v$ takes the form
	\beq\label{e:non.compound.weight.functional.form}
	\Psi_v(\usi_{\delta v},\uL_{\delta v})
	\equiv
	\underbrace{
	\varphi_v(\usi_{\delta v})
	\Bigg\{ \prod_{j=1,2}
	\psi^j(x^j)\Bigg\}
	\Bigg\{
	\prod_{e\in\delta v}
	\prod_{j=1,2}
	(\psi_e)^j((\sigma_e)^j)
	\Bigg\}
	}_{\textup{denote this }
		\psi_v(\usi_{\delta v})}
	\Bigg\{
	\prod_{e\in\delta v}
	\underbrace{\Bigg(
	\psi_e(\bL_e)
	\prod_{j=1,2}
	(\psi_e)^j((\sigma_e)^j \,|\,\bL_e)
	\Bigg)}_{
	\textup{denote this }
	\psi_e(\sigma_e,\bL_e)}
	\Bigg\}\,,\eeq
--- in the above, $\varphi_v(\usi_{\delta v})$ is the indicator of a valid pair coloring $\usi_{\delta v}$ (i.e., the pair version of \eqref{e:color.model.variable.factor}),
and $x\in\set{\minus,\plus,\free}^2$ denotes 
the pair frozen spin corresponding to
$\usi_{\delta v}$. If $\Psi$ is any weight on $U$ of the functional form just described, then $\Psi$ is a Lagrangian weight
for the optimization problem from Proposition~\ref{p:block.update.non.compound}, meaning that $\langle \log\Lambda,\nu\rangle$ is constant over $\nu\in\Judicious_{\DD}(U;\omega_{\delta U})$. Conversely, it is easy to see that any Lagrangian weight can be expressed in this form. Moreover, for our convenience we have chosen weights that are somewhat over-parametrized, since we have 
edge weights $(\psi_e)^j((\sigma_e)^j)$
as well as 
$(\psi_e)^j((\sigma_e)^j\,|\,\bL_e)$.
\end{dfn}

We next give the analogue of Definition~\ref{d:nice.lagrange.defn.weights}:

\begin{dfn}[iterative construction of weights
for a non-compound variable]
\label{d:lagrange.defn.Augmented}
In the setting of Proposition~\ref{p:contraction.for.simple.var},
we again let $U$ be the depth-one tree from Definition~\ref{d:judicious.augmented.alphabet}.
We now define a sequence of weights $\Psi_t$ on $U$
(parametrized as in Definition~\ref{d:param.weights.Augmented} for each $t\ge0$), which will be shown to converge as $t\to\infty$ to the Lagrangian weights $\Psi_\infty\equiv \Psi_{\DD}(U;\omega_{\delta U})$ of Proposition~\ref{p:contraction.for.simple.var}. At the same time we will define messages 
$q_t$ which will converge
as $t\to\infty$ to the \textsc{bp} solution
$q\equiv \BPq(U;\Psi_\infty)$ for the $\Psi_\infty$-weighted
model on $U$. (On each edge $e$ of $U$, the \textsc{bp} messages $\dq_{e,t}$ and $\hq_{e,t}$
will now be probability measures on tuple $(\sigma,\bL)$ where 
$\sigma\in\set{\RYGB}^2$ while $\bL$ ranges over all possible types for the clause incident to $e$.)

To start the construction, let $\Lmstar_U$ denote the single-copy weights on $U$ given by Corollary~\ref{c:clause.bp.weights} (and explicitly constructed in Corollary~\ref{c:clause.bp.weights.explicit}) ---
 similarly to \eqref{e:explicit.Lmstar.T}, if $\uta\equiv\uta_U$ is a single-copy coloring of $U$, then
	\[
	\Lmstar_U(\uta)
	=\overbrace{
	\Bigg\{\lmstar_v(x_v)
	\prod_{e\in\delta v} \lmstar_e(\tau_e)
	\Bigg\}
	}^{\Lmstar_v(\uta_{\delta v})}
	\Bigg\{
	\prod_{e\in\delta U} \lmstar_e(\tau_e)
	\Bigg\}
	\]
where $x_v\in\set{\minus,\plus,\free}$ denotes the frozen spin corresponding to $\uta_{\delta v}$,
and for $e\in\delta U$ we take
$\lmstar_e=\dqbul_e$ as defined by \eqref{e:redistributed.bp.messages}. Next, recalling \eqref{e:pi.DD.first.appearance}, we define $\psistar_e(\bL)\equiv \pi_{\DD}(\bL\,|\,\bt_e)$.
We initialize our construction at $t=0$ with 
	\[
	\Psi_0(\usi,\uL)
	\equiv
	\Bigg\{ \prod_{j=1,2}
	\Lmstar_U(\usi^j)\Bigg\}
	\Bigg\{ \prod_{e\in\delta v}
	\psistar_e(\bL_e)\Bigg\}\,,\]
where $(\usi,\uL)\equiv(\usi_U,\uL_U)$
denotes an augmented pair coloring of $U$.
This fits the functional form prescribed
in Definition~\ref{d:param.weights.Augmented} above, since we can rewrite
	\begin{align*}
	\Psi_0(\usi,\uL)
	&=
	\Psi_{v,0}(\usi_{\delta v},\bL_{\delta v})
	\cdot\Bigg\{
	\prod_{e\in\delta U} 
	\overbrace{\Bigg(
	\prod_{j=1,2}
	\lmstar_e((\sigma_e)^j)
	\Bigg)}^{\psi_{e,0}(\sigma_e,\bL_e)}
	\Bigg\}\,,\\
	\Psi_{v,0}(\usi_{\delta v},\bL_{\delta v})
	&=\underbrace{
	\Bigg\{\prod_{j=1,2}
	\Lmstar_v((\usi_{\delta v})^j)
	\Bigg\}}_{\psi_{v,0}(\usi_{\delta v})}
	\Bigg\{
	\prod_{e\in\delta v}
	\psistar_e(\bL_e)
	\Bigg\}\,,\end{align*}
which is consistent with
\eqref{e:non.compound.weight.functional.form}
if for all $e\in\delta v$ we take
$\psi_{e,0}(\sigma,\bL)
	\equiv\psistar_e(\sigma,\bL)$, where
	\beq\label{e:psistar.joint.defn}
	\psistar_e(\sigma,\bL)
	\equiv\psistar_e(\bL)
	\Bigg\{\prod_{j=1,2}
		(\psistar_e)^j(\sigma^j\,|\,\bL)
		\Bigg\}
	= \psistar_e(\bL)
	= \pi(\bL\,|\,\bt_e)\,.\eeq
Having defined $\Psi_0$, we let $q_0\equiv\BPq(U;\Psi_0)$. 
Recall that $\omega_{\delta U}$ is given, and abbreviate $\Omega\equiv\AUGMENT(\omega)$,
as defined by
\eqref{e:pi.DD.first.appearance}. For $t\ge1$, given $\Psi_{t-1}$ we define updated weights $\Psi_t$ by making the following series of updates, started from the boundary $\delta U$ and working up to the root $v$, then working back down to $\delta U$:
\begin{enumerate}[I.]
\item \bemph{Boundary updates.} For each leaf edge $e=(au)\in \delta U$, we update its weight by setting
	\[
	\psi_{e,t}(\sigma,\bL)
	\equiv
	\dq_{e,t}(\sigma,\bL)
	\equiv
	\f{\Omega_e(\sigma,\bL)}{\hq_{e,t}(\sigma,\bL)}
	\Bigg/
	\Bigg(\sum_{\sigma',\bL'}
	\f{\Omega_e(\sigma',\bL')}
		{\hq_{e,t}(\sigma',\bL')}
	\Bigg)\,.
	\]

\item \bemph{Upward pass through clauses.}
For each clause $a$ in $U$,
update $\hq_{av,t}\equiv\BP_{av}[\dq_t]$.

\item \bemph{Root variable update.}
Recall \eqref{e:augmented.model.factors}
that $\varphi_v(\usi_{\delta v},\uL_{\delta v})$
denotes the indicator of a valid augmented pair coloring of $\delta v$. Suppose inductively that the measure 
	\beq\label{e:var.tuple.measure.augmented}
	\Big(\nu_{\delta v}[\Psi_{v,t-1};\hq_{t-1}]
	\Big)
	(\usi_{\delta v},\uL_{\delta v})
	\equiv
	\f1{\dbz_{v,t-1}}
	\varphi_v(\usi_{\delta v};\uL_{\delta v})
	\Psi_{v,t-1}
	(\usi_{\delta v},\uL_{\delta v})
	\prod_{e\in\delta v}
	\hq_{e,t-1}(\sigma_e,\bL_e)
	\eeq
is fully judicious on $\delta v$. We will show below, in Proposition~\ref{p:ctypes.varupdate}, how to define updated weights $\Psi_{v,t}$ such that
$\nu_{\delta v}[\Psi_{v,t};\hq_t]$
is also fully judicious on $\delta v$.
We then use this to update
$\dq_{va,t}\equiv\BP_{va}[\hq_t;\Psi_{v,t}]$
for all $a\in\pd v$.

\item \bemph{Downward pass through clauses.}
For each clause $a$ in $U$, update
	\[\hq_{au,t}\equiv
	\BP_{au} \bigg[\Big( \dq_{va,t-1},
	(\dq_{wa,t})_{w\in\pd a\setminus \set{u,v}}
	\Big)
	\bigg]
	\]
for each child variable $u\in\pd a\setminus v$. 
This completes the definition of $\Psi_t$ and $q_t$.
\end{enumerate}
In summary, by iterating the above steps, we obtain
$(\Psi_t,q_t)$ for all $t\ge0$ on $T$.
The remainder of this subsection
is devoted to analyzing this iteration under the assumptions of Proposition~\ref{p:contraction.for.simple.var}.
\end{dfn}

\begin{dfn}\label{d:judicious.with.types}
In the setting of Definition~\ref{d:judicious.augmented.alphabet}, let $\nu$ be a probability measure on augmented colorings $(\usi,\uL)$ of $U$. For any edge $e$ of $U$ we let $\nu_e$ be the marginal law of $(\sigma_e,\bL_e)$ under $\nu$. We say $\nu_e$ is \bemph{judicious on average} if
	\[
	\bar{\nu}_e(\sigma_e)
	\equiv \sum_{\bL_e}
	\bar{\nu}_e(\sigma_e,\bL_e)
	\]
is a judicious measure on $\set{\RYGB}^2$, in the sense of \eqref{e:margin.judicious}.
We say that $\nu_e$ is \bemph{conditionally judicious} if for all clause types $\bL_e$ that can appear on edge $e$, the conditional measure $\nu_e(\sigma_e\,|\,\bL_e)$ is a judicious measure on $\set{\RYGB}^2$, 
again in the sense of \eqref{e:margin.judicious}.
Note that \bemph{fully judicious}
(Definition~\ref{d:judicious.augmented.alphabet}) implies conditionally judicious which implies judicious on average.
\end{dfn}

The following extends the notation from Definitions~\ref{d:clause.rel.error}~and~\ref{d:var.rel.error}:

\begin{dfn}\label{d:rel.error.notation.w.types}
Let $\hat{p}$ and $\hq$ be two nonnegative measures on pairs $(\sigma,\bL)$ where $\sigma\in\set{\RYGB}^2$ and $\bL$ ranges over clause types. In what follows we shall write
	\[
	\CRELERR(\hat{p},\hq)
	\le \vec{\ep} = \begin{pmatrix}
	\ep \\ \dot{\ep} \\ \ddot{\ep}
	\end{pmatrix}
	\]
to indicate that
$\crelerr(\hat{p}(\cdot\,|\,\bL),
	\hq(\cdot\,|\,\bL)) \le\vec{\ep}$
for all $\bL$ in the sense of Definition~\ref{d:clause.rel.error}, and that
	\[
	\max_{\bL}\Bigg\{
	\bigg|\f{\hq(\bL)}{\hat{p}(\bL)}-1
	\bigg|\Bigg\}\le \ep\,.
	\]
If $\dot{p}$ and $\dq$
are two probability measures on pairs $(\sigma,\bL)$,
we define $\VRELERR(\dot{p},\dq)$ 
in an analogous fashion, extending the notation from Definition~\ref{d:var.rel.error}.
\end{dfn}

In the following we use
$\prodhQ_e$ to be the probability measure on pairs $(\sigma,\bL)$ such that
	\beq\label{e:prod.hQ.with.L}
	\prodhQ_e(\sigma,\bL)
	=\f{\prodhq_e(\sigma)}{\hat{Z}_e}\eeq
where $\hat{Z}_e$ is the normalization.
The following result will be applied to the analysis of Step~III from Definition~\ref{d:lagrange.defn.Augmented}.
It builds on the analysis of Proposition~\ref{p:var.update} and Corollary~\ref{c:var.update}.

\begin{ppn}\label{p:ctypes.varupdate} Let $v$ be a nice variable with incoming messages $\hat{p}$ (in the augmented pair coloring model). Recalling the notation of Definition~\ref{d:rel.error.notation.w.types},
assume that for all $e\in\delta v$ we have (cf.\ \eqref{e:condition.on.hat.messages})
	\beq\label{e:condition.on.hat.msg.with.L}
	\CRELERR(
		\prodhQ_e
		,\hat{p}_e)
	\le	\f{k^{O(1)}}{2^{k\zeta}}
	\begin{pmatrix}
	2^{-k} \\ 1 \\ 2^k
	\end{pmatrix}\,.\eeq
Suppose $v$ has weight $\Phi_v$, of the functional form \eqref{e:non.compound.weight.functional.form},
such that the measure $\mu\equiv \nu_{\delta v}[\Phi_v;\hat{p}]$
(defined using the notation of
\eqref{e:var.tuple.measure.augmented})
is fully judicious. For the variable weights that do not depend on the clause types, assume that
(cf.\ \eqref{e:assume.variable.weights.close.to.one})
	\[
	\adjustlimits
	\sum_{j=1,2}
	\sum_{x\in\set{\minus,\plus,\free}}
	\Big|\phi^j(x)-1\Big|
	\le \f{k^{O(1)}}{2^{k\zeta}}\,,\quad
	\max_{e\in\delta v}\Bigg\{
	\sum_{j=1,2}
	\Big|(\phi_e)^j(\red)-1\Big|
	\Bigg\}
	\le \f{ k^{O(1)}}{2^{k\zeta}}
	\,.\]
For the variable weights that do depend on the clause types, assume that 
	\beq\label{e:bound.edge.weights.cond.on.L}
	\CRELERR(\psistar_e,\phi_e)
	\le \f{k^{O(1)}}{2^{k\zeta}}
	\begin{pmatrix} 2^{-k} \\ 1 \\ 2^k
	\end{pmatrix}\eeq
for $\psistar_e$ as defined by
\eqref{e:psistar.joint.defn}. Now suppose we have a new set of messages $\hq_e$,
such that
 \eqref{e:condition.on.hat.msg.with.L}
 also holds with $\hq$ in place of $\hat{p}$, such that
	\beq\label{e:crelerr.cond.on.L}
	\CRELERR(\hat{p}_e,\hq_e)
	\le \begin{pmatrix}
	\ep_e\\
	\dot{\ep}_e\\
	\ddot{\ep}_e
	\end{pmatrix}
	\le
	\f{k^{O(1)}}{2^{k\zeta}}
	\begin{pmatrix}
	2^{-k} \\ 1 \\ 2^k
	\end{pmatrix}
	\eeq
for each $e\in\delta v$.
Then there exists a new set of weights
$\Psi_v$, also of the functional form \eqref{e:non.compound.weight.functional.form}, such that
$\nu_{\delta v}[\Psi_v;\hq]$ is fully judicious.
In addition, the clause-independent weights in $\Psi_v$ satisfy the bounds
\eqref{e:defn.of.err.v.notation}
from Proposition~\ref{p:var.update}.
The clause-dependent weights in $\Psi_v$ satisfy
	\begin{align}\nonumber
	\max_{\bL}
	\Bigg\{
	\Bigg|\f{\psi_e(\bL)}{\phi_e(\bL)}-1\Bigg|
	+
	\adjustlimits\max_{j=1,2}
	\max_{\tau\ne\red}
	\Bigg|\f{(\psi_e)^j(\tau\,|\,\bL)}
		{(\phi_e)^j(\tau\,|\,\bL)}
		-1\Bigg|
	\Bigg\}
	&\le
	\ep_e 
	+\f{\dot{\ep}_e}{2^k}
	+\f{\min\set{\ddot{\ep}_e,1}}{2^{k(1+\zeta)}}
	+\f{\err_v}{2^{k(1+\zeta)}}\,,\\
	\adjustlimits
	\max_{\bL}
	\max_{j=1,2}
	\Bigg|\f{(\psi_e)^j(\red\,|\,\bL)}
		{(\phi_e)^j(\red\,|\,\bL)}
		-1\Bigg|
	&\le
	\ep_e 
	+\dot{\ep}_e
	+\f{\min\set{\ddot{\ep}_e,1}}{2^{k\zeta}}
	+\f{\err_v}{2^{k\zeta}}
	\label{e:ctypes.varupdate.MAINBOUND}
	\end{align}
If $\dot{p}_e$ are the original outgoing messages
and $\dq_e$ are the new ones, then
(cf.\ Corollary~\ref{c:var.update})
	\beq\label{e:ctypes.varupdate.OUTMSG}
	\VRELERR(\dot{p}_e,\dot{q}_e)
	\le
	k^{O(1)}
	\begin{pmatrix}
	1\\1\\1\\ 
	2^{-k\zeta} \\
	2^{-k\zeta}
	\end{pmatrix} \err_v
	+k^{O(1)}
	\begin{pmatrix}
	0&0&0\\
	0&1&1\\
	0&1&1\\
	1&2^{-k}&2^{-k}\\
	1&1&1\\
	\end{pmatrix}
	\begin{pmatrix}
	\ep_e\\
	\dot{\ep}_e\\
	\min\set{\ddot{\ep}_e,1}/2^{k\zeta}
	\end{pmatrix}\,,
	\eeq
again using the notation from Definition~\ref{d:rel.error.notation.w.types}.

\begin{proof}
Since the proof is somewhat involved, we divide it into a few numbered parts. We again recall from
 \eqref{e:augmented.model.factors}
that $\varphi_v(\usi_{\delta v},\uL_{\delta v})$
denotes the indicator of a valid augmented pair coloring of $\delta v$. It can be expressed as
	\[
	\varphi_v(\usi_{\delta v},\uL_{\delta v})
	= \varphi_v(\usi_{\delta v})
	\prod_{e\in\delta v}
	\Ind{\bL_e\ni \bt_e}\,,
	\]
where $\varphi_v(\usi_{\delta v})$ is the indicator of a valid pair coloring on $\delta v$ (not augmented with clause types) as defined in \eqref{e:color.model.variable.factor};
and we write $\bL\ni\bt$ to indicate that $\bL_j=\bt$ for $j=j(\bt)$.\smallskip


\noindent
\bemph{\hypertarget{p:ctypes.varupdate.ITERATIVE.CONSTRUCTION}{Part~1}. Iterative construction of $\Psi_v$.}
We first give an iterative definition for a sequence of weights $\Psi_{v,t}$ ($t\ge0$),
all of the functional form \eqref{e:non.compound.weight.functional.form}.
We emphasize that this index $t$ is \bemph{purely local to the proof of this proposition}, and is not the same as the $t$ that indexes the up-and-down passes in
Definition~\ref{d:lagrange.defn.Augmented}. 
We will show in the remainder of the proof that this sequence converges as $t\to\infty$ to the desired weights $\Psi_v=\Psi_{v,\infty}$. Initialize $\Psi_{v,0}\equiv\Phi_v$. For each $t\ge0$, we will update from $\Psi_{v,t}$ to $\Psi_{v,t+1}$ in three stages, summarized by the following table:
	\beq\label{e:update.table}
	\begin{array}{c|ccc}
	& \psi_v(\usi_{\delta v}) 
	& (\psi_e)^j(\sigma^j\,|\,\bL)
	& \psi_e(\bL) \\
	\hline
	\Psi_{v,t} & t & t & t \\
	\Psi_{v,t+1/3}	& t+1 & t & t \\
	\Psi_{v,t+2/3}	& t+1 & t+1 & t 
	\end{array}
	\eeq
--- e.g., the last row of the table indicates that the weight $\Psi_{v,t+2/3}$ 
is defined by
\eqref{e:non.compound.weight.functional.form}
with the $(t+1)$-versions of 
$\psi_v(\usi_{\delta v})$
and $(\psi_e)^j(\sigma^j\,|\,\bL)$
(for both $j=1,2$ and all $e\in\delta v$), but with the $t$-version of $\psi_e(\bL)$ (again for all $e\in\delta v$). For notational convenience,
we also define $\Psi_t\equiv \Phi$ for all $t<0$, and
	\beq\label{e:hq.of.t.notation}
	\hq_{e,t}(\sigma,\bL)
	\equiv\begin{cases}
	\hat{p}_e(\sigma,\bL)
	&\textup{if }t<0\,,\\
	\hq_e(\sigma,\bL)
	&\textup{if }t\ge0\,.\\
	\end{cases}
	\eeq
We then abbreviate $\nu_t\equiv\nu_{\delta v}[\Psi_{v,t};\hq_t]$, using the notation of \eqref{e:var.tuple.measure.augmented}; in particular this means $\nu_t=\nu_{\delta v}[\Phi;\hat{p}]$ for all $t<0$, while $\nu_0=\nu_{\delta v}[\Phi;\hq]$.
For $t>0$, the weights
of \eqref{e:update.table} are defined by the following procedure (with $t$ denoting an integer time from now on):
\begin{enumerate}[A.]
\item From $t$ to $t+1/3$: 
we will suppose inductively that the measure
$\nu_{t-2/3}$ is judicious on average (in the terminology of Definition~\ref{d:judicious.with.types}).
(For $t=0$, recall that we defined
$\nu_{-2/3}\equiv\nu_{\delta v}[\Phi,\hat{p}]$,
which by hypothesis is fully judicious and therefore also judicious on average.)
 At time $t$ (for $t\ge0$) we have
	\[
	\nu_t(\usi_{\delta v},\uL_{\delta v})
	= \f1{\dbz_{v,t}}
	\varphi_v(\usi_{\delta v})
	\psi_{v,t}(\usi_{\delta v})
	\prod_{e\in\delta v}
	\Bigg\{
	\Ind{ \bL_e\ni\bt_e }
	\psi_{e,t}(\sigma_e,\bL_e)
	\hq_e(\sigma_e,\bL_e)
	\Bigg\}\,.
	\]
The marginal law of $\usi_{\delta v}$ under $\nu_t$ can then be expressed as
	\beq\label{e:t.avg.no.longer.judicious.on.avg}
	\bar{\nu}_t(\usi_{\delta v})
	\cong
	\varphi_v(\usi_{\delta v})
	\psi_{v,t}(\usi_{\delta v})
	\prod_{e\in\delta v}
	\avhq_{e,t}(\sigma_e)\,,
	\eeq
where $\avhq_{e,t}$ denotes the message at time $t$ averaged over $\bL$:
	\beq\label{e:defn.avhq.t}
	\avhq_{e,t}(\sigma)
	\equiv
	\f1{\avhz_{e,t}}
	\sum_{\bL}
	\Ind{\bL\ni\bt_e}
	\psi_{e,t}(\sigma,\bL)
	\hq_{e,t}(\sigma,\bL)\,,
	\eeq
where $\hq_{e,t}$ is defined by \eqref{e:hq.of.t.notation}. Since we also defined $\psi_{e,t}\equiv\phi_e$ for negative $t$, this means that $\avhq_{e,t}(\sigma_e)$ is also well-defined for negative $t$. For instance, the marginal law of $\usi_{\delta v}$ at time 
$t-2/3$ is
	\beq\label{e:t.minus.two.thirds.avg}
	\bar{\nu}_{t-2/3}(\usi_{\delta v})
	\cong
	\varphi_v(\usi_{\delta v})
	\psi_{v,t}(\usi_{\delta v})
	\prod_{e\in\delta v}
	\avhq_{e,t-1}(\sigma_e)\,,
	\eeq
including in the case $t=0$. Now, from the inductive assumption, $\nu_{t-2/3}$ is judicious on average, which means that the measure
$\bar{\nu}_{t-2/3}$ in
\eqref{e:t.minus.two.thirds.avg}
is judicious in the sense of Definition~\ref{d:constrained.opt.compound.enclosure}.
Comparing \eqref{e:t.avg.no.longer.judicious.on.avg} with \eqref{e:t.minus.two.thirds.avg},
we see that we can
apply Proposition~\ref{p:var.update}
--- with the variable weight function $\psi_{v,t}(\usi_{\delta v})$
and the two sets of incoming messages
$\avhq_{e,t-1}$ and $\avhq_{e,t}$
--- to define a new variable weight function $\psi_{v,t+1}(\usi_{\delta v})$ such that the measure
	\[
	\bar{\nu}_{t+1/3}(\usi_{\delta v})
	\cong
	\varphi_v(\usi_{\delta v})
	\psi_{v,t+1}(\usi_{\delta v})
	\prod_{e\in\delta v}
	\avhq_{e,t}(\sigma_e)
	\]
is judicious in the sense of Definition~\ref{d:constrained.opt.compound.enclosure},
i.e., such that $\nu_{t+1/3}$ is judicious on average.
Note that, since $t+1/3=(t+1)-2/3$,
this verifies the inductive assumption that
$\nu_{t-2/3}$ is judicious on average.

\item From $t+1/3$ to $t+2/3$:
for all $e\in\delta v$ and both $j=1,2$, update
	\beq\label{e:cond.judicious.update}
	\f{(\psi_{e,t+1})^j
		(\tau\,|\,\bL)}
		{(\psi_{e,t})^j(
		\tau\,|\,\bL)}
	=
	\f{\starpi_e(\tau)}
		{ (\nu_{e,t+1/3})^j(\tau\,|\,\bL) }
	\eeq
for all $\tau\in\set{\RYGB}$.

\item From $t+2/3$ to $t+1$: for all $e\in\delta v$ and both $j=1,2$, update 
	\beq\label{e:clause.frac.update}
	\f{\psi_{e,t+1}(\bL)}{\psi_{e,t}(\bL)}
	= \f{\pi_{\DD}(\bL\,|\,\bt_e)}
		{ \nu_{e,t+2/3}
		(\bL_e) }\,,
	\eeq
using the notation from \eqref{e:pi.DD.first.appearance}.
\end{enumerate}
The rest of the proof is devoted to the analysis of this iterative procedure. For the calculations that follow, it is useful to note that if $\bL\not\ni\bt_e$, then we must have $\psi_{e,t}(\bL)=0$ for all $t$. 
\smallskip

\noindent\bemph{\hypertarget{p:ctypes.varupdate.ERROR.NOTATIONS}{Part~2}. Error notations.}
In the proof we will track the following error quantities for all times $t$. Let $\cdel$ measure the deviation from being conditionally judicious:
for $e\in\delta v$, $j=1,2$, 
and $\tau\in\set{\RYGB}$, let
	\beq\label{e:defn.cdel.kappa}
	(\cdel_{e,t})^j(\tau\,|\,\bL)
	\equiv
	\f{(\nu_{e,t})^j
	( \tau\,|\,\bL) }{\starpi_e(\tau)}
	-1\,.\eeq
Taking the marginal over $\bL$ gives
	\begin{align*}
	(\nu_{e,t})^j(\tau)
	&= \sum_{\bL}
	\nu_{e,t}(\bL)
	\cdot
	(\nu_{e,t})^j
		( \tau\,|\,\bL)
	= \sum_{\bL}
	\nu_{e,t}(\bL)
	\starpi_e(\tau)
	\Bigg\{ 1 + 
	(\cdel_{e,t})^j
		(\tau\,|\,\bL)
	\Bigg\}\\
	&=
	\starpi_e(\tau)
	\Bigg\{
	1 + 
	\underbrace{\bigg[\sum_{\bL}
	\nu_{e,t}(\bL)
	\cdot (\cdel_{e,t})^j(\tau\,|\,\bL)
	\bigg]}_{\textup{denote this }
	(\adel_{e,t})^j(\tau)}
	\Bigg\}\,.
	\end{align*}
Notice also that since $\starpi_e$
and $(\nu_{e,t})^j(\cdot\,|\,\bL)$
are both probability measures over $\set{\RYGB}$, we must have
	\beq\label{e:cdel.avg.to.zero}
	\sum_\tau
	\starpi_e(\tau)
	(\cdel_{e,t})^j(\tau\,|\,\bL)
	=0
	=\sum_\tau
	\starpi_e(\tau)
	(\adel_{e,t})^j(\tau)\,.
	\eeq
Next let $\fdel$ measure the deviations in the clause proportions: for $e\in\delta v$,
	\beq\label{e:fdel.defn}
	\fdel_{e,t}(\bL)
	\equiv
	\f{\pi_{\DD}(\bL\,|\,\bt_e)}
	{\nu_{e,t}(\bL)}-1\,.\eeq
Next, recall the definition \eqref{e:prod.hQ.with.L} of $\prodhQ_e(\sigma,\bL)$, and let
	\beq\label{e:alpha.e.t.defn}
	\alpha_{e,t}(\sigma,\bL)
	\equiv
	\f{\psi_{e,t}(\sigma,\bL)\hq_e(\sigma,\bL)}
	{ \psi_{e,0}(\sigma,\bL) \prodhQ_e(\sigma,\bL) }
	-1
	=\f{\psi_{e,t}(\sigma,\bL)\hq_e(\sigma,\bL)}
	{ \psistar_e(\bL) \prodhq_e(\sigma) /\hat{Z}_e}
	-1\,.\eeq
Lastly, let $\xi$ measure the deviation from being a product measure: for $e\in\delta v$,
	\[
	\xi_{e,t}(\sigma\,|\,\bL)
	\equiv
	\f{\nu_{e,t}(\sigma\,|\,\bL)}
		{\prodom_e(\sigma)}-1
	\]
We will also make some shorthand notation for the maximum absolute values of the above quantities: let
	\beq\label{e:cdel.max.notation}
	\cdel_{e,t}(\tau)
	\equiv
	\max_{\bL,j}\Bigg\{
	\Big|(\cdel_{e,t})^j(\tau\,|\,\bL)\Big|
	\Bigg\}
	\,,\quad
	\begin{pmatrix}
	\cdel_{e,t}\\
	\dot{\cdel}_{e,t}
	\end{pmatrix}
	\equiv\begin{pmatrix}
	\max\set{\cdel_{e,t}(\tau):\tau\ne\red}\\
	\cdel_{e,t}(\red)
	\end{pmatrix}\,,\eeq
and let $\fdel_{e,t}
\equiv \max_{\bL}\set{|\fdel_{e,t}(\bL)|}$.
Finally let
	\beq\label{e:he.max.notation}
	\alpha_{e,t}(\sigma)
	\equiv
	\max_{\bL}\Bigg\{
		\Big|\alpha_{e,t}(\sigma\,|\,\bL)\Big|
			\Bigg\}
	\,,\quad
	\begin{pmatrix}
	\alpha_{e,t}\\
	\dot{\alpha}_{e,t}\\
	\ddot{\alpha}_{e,t}
	\end{pmatrix}
	=\begin{pmatrix}
	\max\set{\alpha_{e,t}(\sigma) : \red[\sigma]=0}\\
	\max\set{\alpha_{e,t}(\sigma) : \red[\sigma]=1}\\
	\max\set{\alpha_{e,t}(\sigma) : \red[\sigma]=2}
	\end{pmatrix}\,,
	\eeq
and make the analogous notation with $\xi$ in place of $\alpha$.
We will show by induction that for all $e\in\delta v$,
	\beq\label{e:rough.inductive.bounds.forall.t}
	\sum_{t\in\mathbb{Z}/3,t\ge1/3}
	\begin{pmatrix}
	\fdel_{e,t}+\cdel_{e,t} + \alpha_{e,t}
		\\
	\dot{\cdel}_{e,t}+\dot{\alpha}_{e,t}
		+\xi_{e,t}
		+ \dot{\xi}_{e,t}\\
	\ddot{\alpha}_{e,t}
		+\ddot{\xi}_{e,t}
	\end{pmatrix}
	\le 
	\f{k^{O(1)}}{2^{k\zeta}}
	\begin{pmatrix} 2^{-k} \\ 1 
		\\ 2^k \end{pmatrix}\eeq
for all $t\ge 1/3$. We also record a calculation that we will use repeatedly in what follows: for any integer $t$, at time $t+1/3$, according to the table \eqref{e:update.table} we have
	\[
	\nu_{t+1/3}(\usi_{\delta v},\uL_{\delta v})
	\cong
	\varphi_v(\usi_{\delta v},\uL_{\delta v})
	\psi_{v,t+1}(\usi_{\delta v})
	\prod_{e\in\delta v}\Bigg\{
	\psi_{e,t}(\sigma_e,\bL_e)
	\hq_{e,t}(\sigma_e,\bL_e)\Bigg\}\,,
	\]
where we recall that $\hq_{e,t}$ is defined by \eqref{e:hq.of.t.notation} for all $t$.
From this it is straightforward to verify that the marginal on an edge $e\in\delta v$ is given by
\beq\label{e:first.defn.of.u}
	\nu_{e,t+1/3}(\sigma,\bL)
	=\f1{\bbz_{e,t+1/3}}
	\psi_{e,t}(\sigma,\bL)
	\hq_{e,t}(\sigma,\bL)
	\overbrace{\Bigg[
	\f1{\dot{z}_{e,t+1/3}}
	\sum_{\usi_{\delta v\setminus e}}
	\varphi_v(\usi_{\delta v})
	\psi_{v,t+1}(\usi_{\delta v})
	\prod_{e'}
	\avhq_{e,t}(\sigma_{e'})
	\Bigg]}^{\textup{denote this }u_{e,t+1/3}(\sigma)}
	\eeq
where $\dot{z}_{e,t+1/3}$ is the normalizing constant that makes $u_{e,t+1/3}$ a probability measure over $\set{\RYGB}^2$, and $\bbz_{e,t+1/3}$ is the normalizing constant that makes $\nu_{e,t+1/3}$ a probability measure over pairs $(\sigma,\bL)$. For integers $t$ we will compare the following measures:
	\begin{align}\nonumber
	\bbz_{e,t-2/3}\nu_{e,t-2/3}(\sigma,\bL)
	&=\hq_{e,t-1}(\sigma,\bL)
		\psi_{e,t-1}(\bL)
		\psi_{e,t-1}(\sigma\,|\,\bL)
		u_{e,t-2/3}(\sigma,\bL)\,,\\
	\nonumber
	\bbz_{e,t-1/3}\nu_{e,t-1/3}(\sigma,\bL)
	&=\hq_{e,t-1}(\sigma,\bL)
		\psi_{e,t-1}(\bL)
		\psi_{e,t}(\sigma\,|\,\bL)
		u_{e,t-1/3}(\sigma,\bL)\,,\\
	\nonumber
	\bbz_{e\bullet,t}
		\nu_{e\bullet,t}(\sigma,\bL)
	&=\hq_{e,t}(\sigma,\bL)
		\psi_{e,t}(\bL)
		\psi_{e,t}(\sigma\,|\,\bL)
		u_{e,t-2/3}(\sigma,\bL)\,,\\
	\nonumber
	\bbz_{e,t}\nu_{e,t}(\sigma,\bL)
	&=\hq_{e,t}(\sigma,\bL)
		\psi_{e,t}(\bL)
		\psi_{e,t}(\sigma\,|\,\bL)
		u_{e,t}(\sigma,\bL)\,,\\
	\bbz_{e,t+1/3}\nu_{e,t+1/3}(\sigma,\bL)
	&=\hq_{e,t}(\sigma,\bL)
		\psi_{e,t}(\bL)
		\psi_{e,t}(\sigma\,|\,\bL)
		u_{e,t+1/3}(\sigma,\bL)\,.
	\label{e:table.of.all.measures}
	\end{align}
The remaining parts of the proof are organized as follows:
\begin{enumerate}[--]
\item
\hyperlink{p:ctypes.varupdate.PRELIM.BOUNDS}{Part~3a}
proves some preliminary bounds.

\item 
\hyperlink{p:ctypes.varupdate.UPDATE.A.GENERAL}{Part~3b} controls $\nu_{e,t+1/3}$ 
in terms of $\nu_{e\bullet,t}$ 
for integers $t\ge0$, in the bounds \eqref{e:fdel.bounds.from.update.A}~and~\eqref{e:cdel.bounds.from.update.A}.
The bounds also depend on the error $\vec{\ep}(t)$ (see \eqref{e:update.A.def.epsilon})
between
$\avhq_{e,t-2/3}$ and $\avhq_{e,t}$.

\item
\hyperlink{p:ctypes.varupdate.UPDATE.B.MSG}{Part~4a} 
controls 
 $\avhq_{e,t-1/3}$ in terms of
 $\nu_{e,t-2/3}$ and $\avhq_{e,t-2/3}$,
in the bound \eqref{e:ctypes.one.two.crelerr}.

\item
\hyperlink{p:ctypes.varupdate.UPDATE.C.MSG}{Part~4b} controls
$\avhq_{e,t}$ in terms of
$\nu_{e,t-1/3}$ and $\avhq_{e,t-1/3}$,
in the bound \eqref{e:error.in.avhq.from.update.C}.
\item \hyperlink{p:ctypes.varupdate.UPDATE.B.MGL}{Part~5a} controls
	$\nu_{e,t-1/3}$
	in terms of $\nu_{e,t-2/3}$
	 for integers $t\ge1$
	 (in the bounds \eqref{e:one.to.two.thirds.clause}
and \eqref{e:ctypes.one.two.cond}).
\item \hyperlink{p:ctypes.varupdate.UPDATE.C.MGL}{Part~5b} compares
	$\nu_{e,t-1/3}$
	with $\nu_{e\bullet,t}$
	and $\nu_{e,t}$ for integers $t\ge1$
	(in the bounds \eqref{e:update.C.single.edge}
and \eqref{e:update.C.single.edge.back.to.bullet}).

\item 
\hyperlink{p:ctypes.varupdate.CONCLUSION}{Part~6} 
combines the results of the preceding parts to show that the iteration of Part~1 converges to the desired weights $\Psi_v\equiv\Psi_{v,\infty}$, and moreover that these weights satisfy the claimed bounds
\eqref{e:defn.of.err.v.notation}~and~\eqref{e:ctypes.varupdate.MAINBOUND}.

\item \hyperlink{p:ctypes.varupdate.OUTMSG}{Part~7}
shows that the messages outgoing from the variable
satisfy the claimed bound
\eqref{e:ctypes.varupdate.OUTMSG}, thereby concluding the proof.

\end{enumerate}\smallskip

\noindent\bemph{\hypertarget{p:ctypes.varupdate.PRELIM.BOUNDS}{Part~3a}. Preliminary bounds.} Recall the definition \eqref{e:alpha.e.t.defn}
of $\alpha_{e,t}$, and observe that
$\alpha_{e,1/3}=\alpha_{e,0}$. It follows from the assumptions 
\eqref{e:condition.on.hat.msg.with.L} and
\eqref{e:bound.edge.weights.cond.on.L} that
	\[
	\Big|\alpha_{e,1/3}(\sigma,\bL)\Big|
	\le
	\f{k^{O(1)}}{2^{k\zeta}}
	\begin{cases}
	2^{-k} & \textup{if $\red[\sigma]=0$,}\\
	1 & \textup{if $\red[\sigma]=1$,}\\
	2^k & \textup{if $\red[\sigma]=2$,}\\
	\end{cases}\]
which verifies that 
$\alpha_{e,1/3}(\sigma)$
 satisfies the bound from \eqref{e:rough.inductive.bounds.forall.t}.
The marginal on $\bL$ at time $t=1/3$ satisfies
	\begin{align}\nonumber
	\nu_{e,1/3}(\bL)
	&\stackrel{\eqref{e:first.defn.of.u}}{=}
	\f1{\bbz_{e,1/3}} \sum_\sigma
		\phi_e(\sigma,\bL)
		\hq_e(\sigma,\bL)
		u_{e,1/3}(\sigma)
	\\
	&\stackrel{\eqref{e:alpha.e.t.defn}}{=}
	\f{\psistar_e(\bL)}{\bbz_{e,1/3}
		\hat{Z}_e}
	\sum_\sigma
	\prodhq_e(\sigma)
	\Bigg(1+\alpha_e(\sigma,\bL)\Bigg)
	u_{e,1/3}(\sigma)
	= \psistar_e(\bL)
	\Bigg\{ 1 + \f{k^{O(1)}}{2^{k(1+\zeta)}}
	\Bigg\}\,,
	\label{e:prod.of.zs.is.close.to.one}
	\end{align}
which proves that $\fdel_{e,1/3}$
also satisfies the bound
from \eqref{e:rough.inductive.bounds.forall.t}.
Next, the marginal on $\sigma$ at time $t=1/3$ is given by
	\beq\label{e:one.third.on.sigma}
	\nu_{e,1/3}(\sigma)
	= \f{\avhz_{e,1/3}}{\bbz_{e,1/3}}
	\overbrace{\Bigg[
	\f1{\avhz_{e,1/3}}
	\sum_{\bL}
	\Ind{\bL\ni\bt_e}
	\phi_e(\sigma,\bL)\hq_e(\sigma,\bL)
	\Bigg]}^{\avhq_{e,0}(\sigma)
		=\avhq_{e,1/3}(\sigma)}
	u_{e,1/3}(\sigma)\,,
	\eeq
where we recall that $\avhq_{e,t}$ is defined by
 \eqref{e:defn.avhq.t}. It follows from 
the assumptions \eqref{e:condition.on.hat.msg.with.L} and \eqref{e:bound.edge.weights.cond.on.L}
that
	\beq\label{e:avhq.rel.err.rough.bound.BASECASE}
	\crelerr(\prodhq_e,\avhq_{e,1/3})
	\le \f{k^{O(1)}}{2^{k\zeta}}
	\begin{pmatrix} 2^{-k} \\ 1 \\ 2^k
		\end{pmatrix}\,.
	\eeq
Substituting
\eqref{e:avhq.rel.err.rough.bound.BASECASE}
into \eqref{e:one.third.on.sigma}
shows that $\nu_{e,1/3}(\sigma)$
is very close to the 
normalization of the 
measure
$\prodhq_e(\sigma) u_{e,1/3}(\sigma)$.
On the other hand, the same must be true for the conditional measure $\nu_{e,1/3}(\sigma\,|\,\bL)$, since
	\begin{align*}
	\nu_{e,1/3}(\sigma\,|\,\bL)
	&\stackrel{\eqref{e:first.defn.of.u}}{=}
	\f{
	\phi_e(\sigma,\bL)\hq_e(\sigma,\bL) u_{e,1/3}(\sigma) }
	{ \bbz_{e,1/3}\nu_{e,1/3}(\bL)}
	\stackrel{\eqref{e:alpha.e.t.defn}}{=}
	\f{\psistar_e(\bL)
	\prodhq_e(\sigma)
	u_{e,1/3}(\sigma)}{ \bbz_{e,1/3} \hat{Z}_e
	\nu_{e,1/3}(\bL)}
	\Bigg\{ 1 + \alpha_e(\sigma,\bL)\Bigg\}\\
	&\stackrel{\eqref{e:prod.of.zs.is.close.to.one}}
		{=}
	\psistar_e(\bL)
	\prodhq_e(\sigma)
	u_{e,1/3}(\sigma)
	\Bigg\{ 1 + \alpha_e(\sigma,\bL)\Bigg\}
	\Bigg\{ 1 + \f{k^{O(1)}}{2^{k(1+\zeta)}}
	\Bigg\}\,.
	\end{align*}
Thus, by comparing both $\nu_{e,1/3}(\sigma)$
and $\nu_{e,1/3}(\sigma\,|\,\bL)$
with the normalization of
$\prodhq_e(\sigma)u_{e,1/3}(\sigma)$, we deduce that
	\beq\label{e:conditional.vs.avg.error.BASECASE}
	\Bigg|
	\f{\nu_{e,1/3}(\sigma\,|\,\bL)}
		{\nu_{e,1/3}(\sigma)}-1\Bigg|
	\le \f{k^{O(1)}}{2^{k\zeta}}
	\begin{cases}
	2^{-k} & \textup{if $\red[\sigma]=0$,}\\
	1 & \textup{if $\red[\sigma]=1$,}\\
	2^k & \textup{if $\red[\sigma]=2$.}\\
	\end{cases}\eeq
Since the measure
$\nu_{e,1/3}$ is judicious by construction, 
this implies that $\cdel_{e,1/3}(\tau)$
satisfies the bounds from \eqref{e:rough.inductive.bounds.forall.t}. 
Finally, we note that the variable \textsc{bp} recursion
together with \eqref{e:avhq.rel.err.rough.bound.BASECASE} implies
	\[\Bigg|
	\f{u_{e,1/3}(\sigma)}
		{\proddq_e(\sigma)}-1\Bigg|
	\le \f{k^{O(1)}}{2^{k\zeta}}
	\]
for all $\sigma\in\set{\RYGB}^2$. Combining this with \eqref{e:condition.on.hat.msg.with.L}
and \eqref{e:bound.edge.weights.cond.on.L} gives
	\[
	\Bigg|
	\f{\nu_{e,1/3}(\sigma\,|\,\bL)}
		{\prodom_e(\sigma)
		}-1\Bigg|\le
	\f{k^{O(1)}}{2^{k\zeta}}
	\times
	\begin{cases}
	1& \textup{if $\red[\sigma]\le1$,}\\
	2^k & \textup{if $\red[\sigma]=2$.}\\
	\end{cases}
	\]
This implies that $\xi_{e,1/3}(\sigma)$ satisfy the bounds from \eqref{e:rough.inductive.bounds.forall.t}.
Thus we have proved \eqref{e:rough.inductive.bounds.forall.t}
for the case $t=1/3$. We note for later use that
the inductive hypothesis \eqref{e:rough.inductive.bounds.forall.t} 
implies that \eqref{e:avhq.rel.err.rough.bound.BASECASE} and \eqref{e:conditional.vs.avg.error.BASECASE}
hold more generally:
	\beq\label{e:avhq.rel.err.rough.bound}
	\crelerr(\prodhq_e,\avhq_{e,t})
	\le \f{k^{O(1)}}{2^{k\zeta}}
	\begin{pmatrix} 2^{-k} \\ 1 \\ 2^k
		\end{pmatrix}
	\eeq
for all $t\ge1/3$, and likewise
	\beq\label{e:conditional.vs.avg.error}
	\Bigg|
	\f{\nu_{e,t}(\sigma\,|\,\bL)}
		{\nu_{e,t}(\sigma)}-1\Bigg|
	\le \f{k^{O(1)}}{2^{k\zeta}}
	\begin{cases}
	2^{-k} & \textup{if $\red[\sigma]=0$,}\\
	1 & \textup{if $\red[\sigma]=1$,}\\
	2^k & \textup{if $\red[\sigma]=2$}\\
	\end{cases}\eeq
for all $t\ge1/3$.\smallskip 

\noindent\bemph{\hypertarget{p:ctypes.varupdate.UPDATE.A.GENERAL}{Part~3b}. Analysis of general applications of update A.} We next analyze the update from $t-2/3$ to $t+1/3$ for integer times $t\ge0$. 
At times $t-2/3$ and $t+1/3$ we have
(cf.\ \eqref{e:first.defn.of.u},
and using the notation from \eqref{e:hq.of.t.notation})
	\begin{align}\nonumber
	\bbz_{e,t-2/3}
	\nu_{e,t-2/3}(\sigma,\bL)
	&= 
	\psi_{e,t-1}(\sigma,\bL)
	\hq_{e,t-1}(\sigma,\bL)
	\overbrace{\Bigg[
	\f1{\dot{z}_{e,t-2/3}}
	\sum_{\usi_{\delta v\setminus e}}
	\varphi_v(\usi_{\delta v})
	\psi_{v,t}(\usi_{\delta v})
	\prod_{e'\in\delta v\setminus e}
	\avhq_{e',t-1}(\sigma_{e'})
	\Bigg]}^{u_{e,t-2/3}(\sigma) }\,,\\
	\bbz_{e,t+1/3}\nu_{e,t+1/3}(\sigma,\bL)
	&= 
	\psi_{e,t}(\sigma,\bL)
	\hq_{e,t}(\sigma,\bL)
	\underbrace{\Bigg[
	\f1{\dot{z}_{e,t+1/3}}
	\sum_{\usi_{\delta v\setminus e}}
	\varphi_v(\usi_{\delta v})
	\psi_{v,t+1}(\usi_{\delta v})
	\prod_{e'\in\delta v\setminus e}
	\avhq_{e',t}(\sigma_{e'})
	\Bigg]}_{u_{e,t+1/3}(\sigma) }\,.
	\label{e:compare.tplusthird.to.tminustwothirds}
	\end{align}
We then estimate the error between $u_{e,t-2/3}$ and $u_{e,t+1/3}$. Recall the notation of Definition~\ref{d:clause.rel.error}, and let 
	\beq\label{e:update.A.def.epsilon}
	\vec{\ep}(t)
	\equiv\begin{pmatrix}
	\ep_e(t)\\
	\dot{\ep}_e(t)\\
	\ddot{\ep}_e(t)
	\end{pmatrix}
	\equiv
	\crelerr(\avhq_{t-2/3},
	\avhq_{t-1/3})
	+\crelerr(\avhq_{t-1/3},
	\avhq_{t})\,.\eeq
(Note that $\avhq_t=\avhq_{t+1/3}$.)
Recall also the notation of Definition~\ref{d:var.rel.error}, and let
	\beq\label{e:delta.vec.of.t}
	\vec{\delta}(t+1/3)
	\equiv
	\begin{pmatrix}
	\delta_e(t)\\
	\dot{\delta}_e(t+1/3)\\
	\ddot{\delta}_e(t+1/3)\\
	\mdel_e(t+1/3)\\
	\mdelred_e(t+1/3)
	\end{pmatrix}
	\equiv\vrelerr(
	u_{e,t-2/3},
	u_{e,t+1/3})\,.
	\eeq
The measure $\nu_{e,t-2/3}$ is judicious on average for all $t\ge0$ (and fully judicious for $t=0$).
It then follows by Corollary~\ref{c:var.update} (whose conditions are satisfied, in view of 
\eqref{e:avhq.rel.err.rough.bound}) that
	\beq\label{e:vec.delta.bound.with.L}
	\vec{\delta}(t+1/3)
	\le
	k^{O(1)}
	\begin{pmatrix}
	1\\1\\1\\ 2^{-k\zeta} \\2^{-k\zeta}
	\end{pmatrix} \err_v(t)
	+ k^{O(1)}
	\begin{pmatrix}
	0&0&0\\
	0&1&1\\
	0&1&1\\
	1&2^{-k}&2^{-k}\\
	1&1&1\\
	\end{pmatrix}
	\begin{pmatrix}
	\ep_e(t)\\
	\dot{\ep}_e(t)\\
	\min\set{\ddot{\ep}_e(t),1}/2^{k\zeta}
	\end{pmatrix}\,,
	\eeq
with $\err_v(t)$ defined analogously to \eqref{e:defn.of.err.v.notation}. Now,
recalling \eqref{e:compare.tplusthird.to.tminustwothirds}, we will consider also an intermediate measure $\nu_{e\bullet,t}$, which we define by
	\beq\label{e:intermediate.bullet.measure}
	\nu_{e\bullet,t}(\sigma,\bL)
	=\f1{\bbz_{e\bullet,t}}
		\psi_{e,t}(\sigma,\bL)
	\hq_{e,t}(\sigma,\bL)
	 u_{e,t-2/3}(\sigma)\,.\eeq
Let $\adel_{e\bullet,t}$,
$\cdel_{e\bullet,t}$, etc.\
be defined analogously to the quantities in 
\hyperlink{p:ctypes.varupdate.ERROR.NOTATIONS}{Part~2}
with $\nu_{e\bullet,t}$
in place of $\nu_{e,t}$. 
In this part of the proof we shall abbreviate
(recalling \eqref{e:cdel.max.notation} and \eqref{e:he.max.notation})
	\beq\label{e:max.time.t.notation}
	\begin{pmatrix}
	\cdel_{e\bullet}(\tau)\\
	\alpha_{e\bullet}(\sigma)
	\end{pmatrix}
	\equiv\begin{pmatrix}
	\max_{\bL,j}\set{|(
		\cdel_{e\bullet,t}
		)^j(\tau\,|\,\bL)|}\\
	\max_{\bL}\set{|
		\alpha_{e\bullet,t}
		(\sigma\,|\,\bL)|}
	\end{pmatrix}\,,
	\eeq
Comparing with $\nu_{e,t+1/3}$ gives
	\beq\label{e:Upsilon.t.in.terms.of.u}
	\gamma_{e,t}(\sigma)
	\equiv
	\f{\nu_{e,t+1/3}(\sigma,\bL)}
	{\nu_{e\bullet,t}(\sigma,\bL)} -1
	= 
	\f{ \bbz_{e\bullet,t}}{\bbz_{e,t+1/3}}
	\cdot
	\f{ u_{e,t+1/3}(\sigma)}
	{ u_{e,t-2/3}(\sigma)} - 1\,,
	\eeq
where we emphasize that this error does not depend on $\bL$. As a result, marginally on $\sigma$ we also have
	\beq\label{e:marginal.relation.for.gamma}
	\gamma_{e,t}(\sigma)
	=
	\f{\nu_{e,t+1/3}(\sigma)}
	{\nu_{e\bullet,t}(\sigma)} -1\,.\eeq
Substituting \eqref{e:vec.delta.bound.with.L} into the expression \eqref{e:Upsilon.t.in.terms.of.u} for $\gamma_{e,t}$ gives
	\beq\label{e:bound.Upsilon.in.terms.of.eps}
	\Big|\gamma_{e,t}(\sigma)\Big|
	\le 
	k^{O(1)}\Bigg\{
	\err_v(t)
	+ \Ind{\red[\sigma]\ge1}\Bigg(
	\dot{\ep}_e(t)
	+\f{\min\set{\ddot{\ep}_e(t),1}}
		{2^{k\zeta}}
	\Bigg)
	\Bigg\}\,.
	\eeq
We then compare the marginal laws of $(\sigma^1,\bL)$ 
under $\nu_{e\bullet,t}$ and $\nu_{e,t+1/3}$:
by the definition of $\gamma_{e,t}$, we have
	\begin{align}\nonumber
	&\nu_{e,t+1/3}(\sigma^1,\bL)
	-\nu_{e\bullet,t}(\sigma^1,\bL)
	=
	\sum_{\sigma^2}
	\nu_{e\bullet,t}(\sigma,\bL)
	\gamma_{e,t}(\sigma)\\
	&=
	\nu_{e\bullet,t}(\bL)
	\starpi_e(\sigma^1)
	\Bigg\{
	\underbrace{
	\Bigg[
	\sum_{\sigma^2}
	\f{\nu_{e\bullet,t}(\sigma)}
		{\starpi_e(\sigma^1)}
	\gamma_{e,t}(\sigma)
	\Bigg]
	}_{\textup{denote this }
	(X_{e,t})^1(\sigma^1)}
	+
	\underbrace{\Bigg[
	\sum_{\sigma^2}
	\f{\nu_{e\bullet,t}(\sigma\,|\,\bL)
	-\nu_{e\bullet,t}(\sigma)}{\starpi_e(\sigma^1)}
	\gamma_{e,t}(\sigma)
	\Bigg]}_{\textup{denote this }
		(Y_{e,t})^1
		(\sigma^1\,|\,\bL )}
	\Bigg\}\,.
	\label{e:zero.to.one.third.errors}
	\end{align}
We note that the $X$ term
in \eqref{e:zero.to.one.third.errors} can be simplified
using \eqref{e:marginal.relation.for.gamma}
and the fact that $\nu_{e,t+1/3}$ is judicious on average:
	\[
	(X_{e,t})^1(\sigma^1)
	=\sum_{\sigma^2}	
	\f{
	\nu_{e,t+1/3}(\sigma)
	-\nu_{e\bullet,t}(\sigma)
	}{\starpi_e(\sigma^1)}
	= 1 -
	\f{(\nu_{e\bullet,t})^1(\sigma^1)
	}{\starpi_e(\sigma^1)}
	=- (\adel_{e\bullet,t})^1(\sigma^1)\,.
	\]
It follows from the definition of $\adel$ that for all $\tau\in\set{\RYGB}$ we have
	\[
	\Big|
	(X_{e,t})^1(\tau)\Big|
	= \Big|
	(\adel_{e\bullet,t})^1(\tau)
	\Big|
	\le
	\max_{\bL}
	\Bigg\{\Big|
	(\cdel_{e\bullet,t})^1(\tau\,|\,\bL)
	\Big|\Bigg\}
	=\cdel_{e\bullet}(\tau)\,.
	\]
Summing \eqref{e:zero.to.one.third.errors} over $\sigma^1=\tau\in\set{\RYGB}$
and recalling \eqref{e:cdel.avg.to.zero} gives
	\beq\label{e:zero.to.one.third.errors.mgl.on.L}
	\nu_{e,t+1/3}(\bL)
	-\nu_{e\bullet,t}(\bL)
	= \nu_{e\bullet,t}(\bL)
	\Bigg\{\sum_{\tau}
	\starpi_e(\tau)
	(Y_{e,t})^1(\tau\,|\,\bL )
	\Bigg\}\,.
	\eeq
Now, the $Y$ term in \eqref{e:zero.to.one.third.errors} can be bounded as 
	\[\Big|(Y_{e,t})^1
		(\sigma^1\,|\,\bL)\Big|
	\le \sum_{\sigma^2}
	\f{\nu_{e\bullet,t}(\sigma)}{\starpi_e(\sigma^1)}
	\Bigg|
	\f{\nu_{e\bullet,t}(\sigma\,|\,\bL)}
	{\nu_{e\bullet,t}(\sigma)}-1
	\Bigg|
	\gamma_{e,t}(\sigma)\,.\]
It follows from \eqref{e:conditional.vs.avg.error} and \eqref{e:bound.Upsilon.in.terms.of.eps} that 
if $\sigma^1\ne\red$,
	\[
	\Big|(Y_{e,t})^1
		(\sigma^1\,|\,\bL)\Big|
	\le
	\f{k^{O(1)}}{2^{k(1+\zeta)}}
	\bigg(\gamma_{e,t}+\dot{\gamma}_{e,t}
		\bigg)
	\le
	\f{k^{O(1)}}{2^{k(1+\zeta)}}
	\Bigg\{
	\err_v(t) + \dot{\ep}_e(t)
		+ \f{\min\set{\dot{\ep}_e(t),1}}
			{2^{k\zeta}}
	\Bigg\}\,.
	\]
If $\sigma^1=\red$ then we recall from
\eqref{e:table.of.all.measures}
 that $\nu_{e\bullet,t}\cong
 \hq_{e,t} \psi_{e,t} u_{e\bullet,t}$,
 from which it follows that
 $\nu_{e\bullet,t}(\red\red\,|\,\bL) \le 
 k^{O(1)}/2^{-(1+\zeta)}$. 
It then follows that
	\begin{align*}
	\Big|(Y_{e,t})^1
		(\sigma^1\,|\,\bL)\Big|
	&\le
	\f{k^{O(1)}}{2^{k\zeta}}
	\dot{\gamma}_{e,t}
	+O\Bigg(
	\f{\max_{\bL}\set{
	\nu_{e\bullet,t}(\red\red\,|\,\bL)}}
	{\starpi_e(\red)}
	\ddot{\gamma}_{e,t}
	\Bigg)
	\le 
	\f{k^{O(1)}}{2^{k\zeta}}
	\bigg(
	\dot{\gamma}_{e,t}
	+\ddot{\gamma}_{e,t}\bigg)
	\\
	&\le
	\f{k^{O(1)}}{2^{k\zeta}}
	\Bigg\{
	\err_v(t) + \dot{\ep}_e(t)
		+ \f{\min\set{\dot{\ep}_e(t),1}}
			{2^{k\zeta}}
	\Bigg\}
	\end{align*}
Substituting these bounds for $Y$ back into
\eqref{e:zero.to.one.third.errors.mgl.on.L} gives, for $t\ge1$, 
	\beq\label{e:fdel.bounds.from.update.A}
	\fdel_{e,t+1/3}
	= \max_{\bL}\Bigg\{ \Big|
	\fdel_{e,t+1/3}(\bL)\Big|\Bigg\}
	\le 
	\fdel_{e\bullet,t}
	+\f{k^{O(1)}}{2^{k(1+\zeta)}}
	\Bigg\{
	\err_v(t) + \dot{\ep}_e(t)
		+ \f{\min\set{\dot{\ep}_e(t),1}}
			{2^{k\zeta}}
	\Bigg\}\,.
	\eeq
Substituting the bounds for $X$ and $Y$ back into \eqref{e:zero.to.one.third.errors},
and combining with \eqref{e:zero.to.one.third.errors.mgl.on.L}, gives
	\begin{align}
	\nonumber
	\cdel_{e,t+1/3}
	&\le 
	\fdel_{e\bullet,t}
	+\cdel_{e\bullet,t}
	+\f{k^{O(1)}}{2^{k(1+\zeta)}}
	\Bigg\{
	 \err_v(t)+\dot{\ep}_e(t)
	+\f{\min\set{\ddot{\ep}_e(t),1}}
		{2^{k\zeta}}
	\Bigg\}\,,\\
	\dot{\cdel}_{e,t+1/3}
	&\le
	\fdel_{e\bullet,t}
	+\cdelred_{e\bullet,t}+
	\f{k^{O(1)}}{2^{k\zeta}}
	\Bigg\{
	 \err_v(t)
	+\dot{\ep}_e(t)
	+\f{\min\set{\ddot{\ep}_e(t),1}}
		{2^{k\zeta}}
	\Bigg\}\,.
	\label{e:cdel.bounds.from.update.A}
	\end{align}
for all $t\ge0$.
We next turn to the question of bounding $\vec{\ep}(t)$ (as defined by \eqref{e:update.A.def.epsilon}) for $t\ge1$.\smallskip

\noindent\bemph{\hypertarget{p:ctypes.varupdate.UPDATE.B.MSG}{Part~4a}. Errors in averaged messages incurred by update B.} We now estimate the error between the averaged messages 
$\avhq_{e,t+1/3}$ and $\avhq_{e,t+2/3}$ for $t\ge0$. Between times $t+1/3$ and $t+2/3$ we only make the update
\eqref{e:cond.judicious.update}. In this part of the proof we will fix $e\in\delta v$ and abbreviate
(recalling
\eqref{e:cdel.max.notation} and \eqref{e:he.max.notation})
	\[
	\begin{pmatrix}
	\cdel_e(\tau) \\ \alpha_e(\sigma)
	\end{pmatrix}
	\equiv
	\begin{pmatrix}
	\cdel_{e,t+1/3}(\tau) \\ 
	\alpha_{e,t+1/3}(\sigma)
	\end{pmatrix}
	=\begin{pmatrix}
	\cdel_{e,t+1/3}(\tau) \\ 
	\alpha_{e,t}(\sigma)
	\end{pmatrix}
	\,.
	\]
By the definition
\eqref{e:defn.cdel.kappa} of $\cdel$, it holds for all $\tau\in\set{\RYGB}$ that
	\beq\label{e:orth}
	0=
	\sum_{\bL}\nu_{e,t+1/3}(\bL)
	\Bigg\{
	(\nu_{e,t+1/3})^j(\tau\,|\,\bL)
	-\starpi_e(\tau)\Bigg\}
	=\sum_{\bL}\nu_{e,t+1/3}(\bL)
	\starpi_e(\tau)
	(\cdel_{e,t+1/3})^j(\tau\,|\,\bL)\,,
	\eeq
where the first equality holds because $\nu_{e,t+1/3}$ is judicious on average. We now turn to comparing the messages. At time $t+1/3$ we have
$\avhq_{e,t+1/3}=\avhq_{e,t}$, defined by
 \eqref{e:defn.avhq.t}. More explicitly, it follows using
\eqref{e:first.defn.of.u} that
	\beq\label{e:relation.of.normalizing.constants}
	\avhz_{e,t}\avhq_{e,t}(\sigma)
	=\sum_{\bL}
	\psi_{e,t}(\sigma,\bL)
	\hq_e(\sigma,\bL)
	= \f{\bbz_{e,t+1/3}
		\nu_{e,t+1/3}(\sigma)
		}{u_{e,t+1/3}(\sigma)}
	\,,\eeq
On the other hand, at time $t+2/3$ we have (by Taylor expansion)
	\begin{align}\nonumber
	&\avhz_{e,t+2/3}\avhq_{e,t+2/3}(\sigma)
	=\sum_{\bL}
	\psi_{e,t}(\sigma,\bL)
	\hq_e(\sigma,\bL)
	\prod_{j=1,2}
	\f1{1 +( \cdel_{e,t+1/3})^j(\sigma^j\,|\,\bL)
	}\\
	&\qquad
	=\avhz_{e,t}
	\Bigg\{
	\avhq_{e,t}(\sigma)
	\bigg\{
	1 + O \bigg(
		\sum_{j=1,2} \cdel_e(\sigma^j)^2
			\bigg)
	\bigg\}
	-\sum_{j=1,2}
	\underbrace{\Bigg[
	\sum_{\bL}
	\f{\psi_{e,t}(\sigma,\bL)
	\hq_e(\sigma,\bL)}{\avhz_{e,t}}
	( \cdel_{e,t+1/3})^j(\sigma^j\,|\,\bL)
	\Bigg]}_{\textup{denote this }K^j(\sigma)}
	\Bigg\}
	\label{e:K.j.defn}
	\end{align} 
From the definition of $K^j(\sigma)$ we can easily deduce the bound 
	\beq\label{e:first.bound.K.j}
	\Big|K^j(\sigma)\Big|
	\le 
	O\Bigg(
	\f{\cdel_e(\sigma^j)}{\avhz_{e,t}\hat{Z}_e}
	\avhq_{e,t}(\sigma)
	\sum_{\bL} \psistar_e(\bL)
	\Bigg)
	= O\bigg(
	\avhq_{e,t}(\sigma)
	\cdot \cdel_e(\sigma^j) \bigg)\,,\eeq
where the last step uses that
$\avhz_{e,t}\hat{Z}_e=\Theta(1)$ by a similar calculation as in \eqref{e:prod.of.zs.is.close.to.one}.
On the other hand, we also have
	\begin{align}\nonumber
	\Big|K^j(\sigma)\Big|
	&=\Bigg|
	\f{\prodhq_e(\sigma)}{\hat{Z}_e}
	\sum_{\bL}\psistar_e(\bL)
	\Bigg(1 + 
	\alpha_{e,t}(\sigma,\bL)
		\Bigg)
	( \cdel_{e,t+1/3})^j(\sigma^j\,|\,\bL)
	\Bigg|\\
	&\stackrel{\eqref{e:orth}}{\le}
	O\bigg( \prodhq_e(\sigma) \cdot \cdel_e(\sigma^j) \cdot \alpha_e(\sigma) \bigg)\,.
	\label{e:second.bound.K.j}
	\end{align}
Combining \eqref{e:first.bound.K.j} with \eqref{e:second.bound.K.j} gives 
	\[
	K^j(\sigma)
	= O\Bigg(
	\avhq_{e,t}(\sigma)\cdot
	\cdel_e(\sigma^j)\cdot
	\min\bigg\{
	\f{\alpha_e(\sigma) \prodhq_e(\sigma)}
		{\avhq_{e,t}(\sigma)},
	1\bigg\}
	\Bigg)\,.
	\]
Substituting this bound into \eqref{e:K.j.defn} gives 
	\begin{align*}
	\avhz_{e,t+2/3}\avhq_{e,t+2/3}(\sigma)
	&= \avhz_{e,t}\avhq_{e,t}(\sigma)
	\Bigg\{
	1+\sum_{j=1,2}\cdel_e(\sigma^j)
	\cdot \min\bigg\{
		\f{\alpha_e(\sigma) \prodhq_e(\sigma)}
		{\avhq_{e,t}(\sigma)},1\bigg\}
	\Bigg\}\,.
	\end{align*}
It then follows using the hypothesis 
\eqref{e:rough.inductive.bounds.forall.t} 
 that
	\[\Bigg|
	\f{\avhz_{e,t+2/3}\avhq_{e,t+2/3}(\sigma)}
	{\avhz_{e,t}\avhq_{e,t}(\sigma)}-1\Bigg|
	\le
	k^{O(1)}
	\begin{cases} \DS
	\f{\cdel}{2^{k(1+\zeta)}}
	&\textup{if }\red[\sigma]=0\,,\medskip\\
	\DS
	\f{\cdel+\cdelred}
		{2^{k\zeta}}
	&\textup{if }\red[\sigma]=1\,,\medskip\\
	\DS \cdelred
	&\textup{if }\red[\sigma]=2\,.
	\end{cases}
	\]
Altogether, the error in averaged messages
between times $t+1/3$
and $t+2/3$ satisfies
	\beq\label{e:ctypes.one.two.crelerr}
	\crelerr
	(\avhq_{e,t+1/3},
	\avhq_{e,t+2/3})
	=\crelerr
	(\avhq_{e,t},
	\avhq_{e,t+2/3})
	\le 
	\f{k^{O(1)}}{2^{k\zeta}}
	\begin{pmatrix}
	2^{-k} & 2^{-k}\\
	1 & 1 \\
	2^{-k}&2^{k\zeta}
	\end{pmatrix}
	\begin{pmatrix}
	\cdel_{e,t+1/3}\\
	\cdelred_{e,t+1/3}
	\end{pmatrix}\,.
	\eeq
Substituting \eqref{e:ctypes.one.two.crelerr}
into the definition of $u$ from \eqref{e:first.defn.of.u} gives
	\beq\label{e:ctypes.one.two.vrelerr}
	\Bigg|
	\f{u_{e,t+2/3}(\sigma)}
		{u_{e,t+1/3}(\sigma)}
	-1\Bigg|
	\le
	\f{k^{O(1)}}{2^{k(1+\zeta)}}
	\Bigg\{
	\underbrace{\sum_{e\in\delta v}
	\cdel_{e,t+1/3}
	}_{\textup{denote this }\bm{\cdel}_{t+1/3}}
	+\underbrace{
	\sum_{e\in\delta v}
	\cdelred_{e,t+1/3}
	}_{\textup{denote this }\bm{\cdelred}_{t+1/3}}
	\Bigg\}
	\equiv
	\f{k^{O(1)}
		(\bm{\cdel}_{t+1/3}+\bm{\cdelred}_{t+1/3})
	}{2^{k(1+\zeta)}}
	\eeq
for all $\sigma\in\set{\RYGB}^2$.\smallskip

\noindent\bemph{\hypertarget{p:ctypes.varupdate.UPDATE.C.MSG}{Part~4b}. Errors in averaged messages incurred by update C.} By a similar (but somewhat simpler) calculation, we can also estimate the error between the averaged messages $\avhq_{e,t+2/3}$ and $\avhq_{e,t+1}$ for $t\ge0$. Between times $t+2/3$ and $t+1$ we only make the update \eqref{e:clause.frac.update}. 
In this part of the proof we will fix $e\in\delta v$ and abbreviate
	\[
	\begin{pmatrix}
	\cdel_e(\tau) \\ \alpha_e(\sigma) \\ \fdel_e
	\end{pmatrix}
	\equiv
	\begin{pmatrix}
	\cdel_{e,t+2/3}(\tau) \\ 
	\alpha_{e,t+2/3}(\sigma) \\ 
	\fdel_{e,t+2/3}
	\end{pmatrix}\,.
	\]
By the definition \eqref{e:fdel.defn} of $\fdel$, we have
	\beq\label{e:fdel.orth}	0=
	\sum_{\bL}\Bigg\{
	\pi_{\DD}(\bL\,|\,\bt_e)
	-\nu_{e,t+2/3}(\bL)
	\Bigg\}
	=\sum_{\bL}
	\pi_{\DD}(\bL\,|\,\bt_e)
	\Bigg\{
	\fdel_{e,t+2/3}(\bL)
	+O\Big((\fdel_e)^2\Big)\Bigg\}\,.\eeq
We now turn to the comparison of the messages:
by the update rule \eqref{e:clause.frac.update},
	\begin{align*}
	\avhz_{e,t+1}\avhq_{e,t+1}(\sigma)
	&= \sum_{\bL}
	\psi_{e,t+2/3}(\sigma,\bL)
	\hq_e(\sigma,\bL)
	\bigg\{1 + \fdel_{e,t+2/3}(\bL) \bigg\} \\
	&=\avhz_{e,t+2/3}
	\Bigg\{\avhq_{e,t+2/3}(\sigma)
	+\underbrace{
	\Bigg[\sum_{\bL}
	\f{\psi_{e,t+2/3}(\sigma,\bL)
	\hq_e(\sigma,\bL)}{\avhz_{e,t+2/3}}
	\fdel_{e,t+2/3}(\bL) 
	\Bigg]}_{\textup{denote this }R(\sigma)}\Bigg\}\,.
	\end{align*}
From the definition of $R(\sigma)$ we can easily deduce the bound
	\[
	\Big|R(\sigma)\Big|
	\le
	O\bigg( \avhq_{e,t+2/3}(\sigma)
		\cdot \fdel_e\bigg)\,.
	\]
On the other hand we also have
	\begin{align*}
	R(\sigma)
	&=
	\f{\prodhq_e(\sigma)
	}{\avhz_{e,t+2/3} \hat{Z}_e}
	\sum_{\bL}
	\psistar_e(\bL)
	\bigg(1 + \alpha_{e,t+1/3}(\sigma,\bL)\bigg)
	\fdel_{e,t+2/3}(\bL) \\
	&=
	O\Bigg(
	\prodhq_e(\sigma)\cdot
	\fdel_e 
	\cdot\Big(\fdel_e+\alpha_e(\sigma)\Big)\Bigg)\,,
	\end{align*}
where the last step uses \eqref{e:fdel.orth}.
Substituting back into the preceding calculation gives
	\[
	\avhz_{e,t+1}\avhq_{e,t+1}(\sigma)
	=\avhz_{e,t+2/3}\avhq_{e,t+2/3}(\sigma)
	\Bigg\{
	1 + O\bigg(
	\fdel_e \cdot
	\min\bigg\{
	\f{(\fdel_e+\alpha_e(\sigma))\prodhq_e(\sigma)}
		{\avhq_{e,t+2/3}(\sigma)} ,1
	\bigg\}\bigg)
	\Bigg\}\,.\]
It follows using the inductive bounds \eqref{e:rough.inductive.bounds.forall.t} that
	\beq\label{e:error.in.avhq.from.update.C}
	\crelerr(\avhq_{e,t+2/3},\avhq_{e,t+1})
	\le 
	\f{k^{O(1)}}{2^{k\zeta}}
	\fdel_{e,t+2/3}
	\begin{pmatrix} 2^{-k} \\ 1 \\ 2^{k\zeta}
	\end{pmatrix}\,.
	\eeq
Substituting this again into the definition \eqref{e:first.defn.of.u} of $u$ gives
	\beq\label{e:ctypes.two.three.vrelerr}
	\Bigg|
	\f{u_{e,t+1}(\sigma)}
	{u_{e,t+2/3}(\sigma)}-1
	\Bigg|
	\le
	\f{k^{O(1)}}{2^{k(1+\zeta)}}
	\sum_{e\in\delta v}
	\fdel_{e,t+2/3}
	\equiv
	\f{k^{O(1)} }{2^{k(1+\zeta)}}
		\bm{\fdel}_{t+2/3}
	\eeq
for all $\sigma\in\set{\RYGB}^2$.\smallskip
	
\noindent\bemph{\hypertarget{p:ctypes.varupdate.UPDATE.B.MGL}{Part 5a}. Errors in marginals incurred by update B.} We now estimate the change in the edge marginals between times $t+1/3$ and $t+2/3$, resulting from update \eqref{e:cond.judicious.update}. To this end we introduce an intermediate measure $\nu_{e,t+1/2}$, in which we use the new weight $\psi_{e,t+1}(\sigma\,|\,\bL)$ but the old message $u_{e,t+1/3}$. Thus, recalling \eqref{e:first.defn.of.u}, we have
	\begin{align*}
	\bbz_{e,t+1/3}\nu_{e,t+1/3}(\sigma,\bL)
	&= 
	\hq_e(\sigma,\bL)\psi_{e,t}(\bL)
	\psi_{e,t}(\sigma\,|\,\bL)
	u_{e,t+1/3}(\sigma)\,,\\
	\bbz_{e,t+1/2}\nu_{e,t+1/2}(\sigma,\bL)
	&= 
	\hq_e(\sigma,\bL)\psi_{e,t}(\bL)
	\psi_{e,t+1}(\sigma\,|\,\bL)
	u_{e,t+1/3}(\sigma)\,,\\
	\bbz_{e,t+2/3}\nu_{e,t+2/3}(\sigma,\bL)
	&=\hq_e(\sigma,\bL) \psi_{e,t}(\bL)
	\psi_{e,t+1}(\sigma\,|\,\bL)
	u_{e,t+2/3}(\sigma)\,.
	\end{align*}
We first estimate the error between times $t+1/3$ and $t+2/3$. For this purpose, note that \eqref{e:defn.cdel.kappa} implies
(similarly to the relation \eqref{e:cdel.avg.to.zero} that was used earlier) that for all $\bL$,
	\begin{align}\nonumber
	0&=\sum_\tau
	\Bigg\{ (\nu_{e,t+1/3})^j(\tau\,|\,\bL)
	-\starpi_e(\tau)\Bigg\}
	=\sum_\tau\Bigg\{
	(\nu_{e,t+1/3})^j(\tau\,|\,\bL)
	-\f{ (\nu_{e,t+1/3})^j(\tau\,|\,\bL)}
		{1 + (\cdel_{e,t+1/3})^j(\tau\,|\,\bL)}
	\Bigg\} \\
	&=\sum_\tau
	(\nu_{e,t+1/3})^j(\tau\,|\,\bL)
	(\cdel_{e,t+1/3})^j(\tau\,|\,\bL)
	+ O\Bigg(\sum_\tau
		\starpi_e(\tau) \cdel_{e,t+1/3}(\tau)^2\Bigg)\,.
		\label{e:orth.sigma}
	\end{align}
It follows using \eqref{e:orth.sigma} that the marginal error on $\bL$ between times $t+1/3$ and $t+1/2$ is given by
	\begin{align}\nonumber
	\bbz_{e,t+1/2}\nu_{e,t+1/2}(\bL)
	&=\sum_\sigma
	\bbz_{e,t+1/3}\nu_{e,t+1/3}(\sigma,\bL)
	\prod_{j=1,2}
	\f1{1+(\cdel_{e,t+1/3})^j(\sigma^j\,|\,\bL)} \\
	&=\bbz_{e,t+1/3}\nu_{e,t+1/3}(\bL)
	\Bigg\{1 +
	O\Bigg(\sum_\tau
		\starpi_e(\tau) \cdel_{e,t+1/3}
		(\tau)^2\Bigg)\Bigg\}\,.
	\label{e:one.third.to.one.half.rho.comparison}
	\end{align}
We note the error term can be simplified
using \eqref{e:rough.inductive.bounds.forall.t} as
	\[
	\sum_\tau
		\starpi_e(\tau) \cdel_{e,t+1/3}
		(\tau)^2
	\le k^{O(1)}\Bigg(
	(\cdel_{e,t+1/3})^2
	+\f{(\cdelred_{e,t+1/3})^2}{2^k}
	\Bigg)
	\le
	\f{k^{O(1)}
		(\cdel_{e,t+1/3}+\cdelred_{e,t+1/3})}
		{ 2^{k(1+\zeta)}}\,.
	\]
The error between $u_{e,t+1/3}$ and $u_{e,t+2/3}$
is bounded by \eqref{e:ctypes.one.two.vrelerr}, and this gives the marginal error on $\bL$ between times $t+1/2$ and $t+2/3$. Combining these bounds gives altogether
	\beq\label{e:one.to.two.thirds.clause}
	\fdel_{e,t+2/3}
	-\fdel_{e,t+1/3}
	\le
	\Bigg|\f{\nu_{e,t+2/3}(\bL)}
	{\nu_{e,t+1/3}(\bL)}-1\Bigg|
	\le \f{k^{O(1)}
		(\bm{\cdel}_{t+1/3}+\bm{\cdelred}_{t+1/3})
	}{2^{k(1+\zeta)}}\,.
	\eeq
Next we argue that the update \eqref{e:cond.judicious.update} in fact brings the measures closer to being conditionally judicious. Indeed, focusing on the first copy $j=1$,
we can express
	\[
	\f{\bbz_{e,t+1/2}
	\nu_{e,t+1/2}(\sigma,\bL)}
	{\bbz_{e,t+1/3}\nu_{e,t+1/3}(\bL)}
	=
	\f{
	\starpi_e(\sigma^1)}
	{ (\nu_{e,t+1/3})^1(\sigma^1\,|\,\bL) }
	\f{\nu_{e,t+1/3}(\sigma\,|\,\bL)}
	{1+(\cdel_{e,t+1/3})^2(\sigma^2\,|\,\bL)}\,,
	\]
and summing over $\sigma^2$ gives
	\begin{align*}
	&\f{\bbz_{e,t+1/2}
	\nu_{e,t+1/2}(\sigma^1,\bL)}
	{\bbz_{e,t+1/3}\nu_{e,t+1/3}(\bL)}
	=\f{\starpi_e(\sigma^1)}
	{ (\nu_{e,t+1/3})^1(\sigma^1\,|\,\bL) }
	\sum_{\sigma^2}
	\nu_{e,t+1/3}(\sigma\,|\,\bL)
	\Bigg\{
	1 + O(\cdel_{e,t+1/3}(\sigma^2)^2) - 
	(\cdel_{e,t+1/3})^2(\sigma^2\,|\,\bL)
	\Bigg\}\\
	&\qquad=
	\starpi_e(\sigma^1)
	\Bigg\{
	1+ 
	O\Bigg(\sum_\tau
		\starpi_e(\tau) \cdel_{e,t+1/3}
		(\tau)^2\Bigg)
	-\sum_{\sigma^2}
	\f{
	\nu_{e,t+1/3}(\sigma\,|\,\bL)}
	{ (\nu_{e,t+1/3})^1(\sigma^1\,|\,\bL) }
	(\cdel_{e,t+1/3})^2(\sigma^2\,|\,\bL)
	\Bigg\}\\
	&\qquad=
	\starpi_e(\sigma^1)
	\Bigg\{
	1+ 
	O\Bigg(\sum_\tau
		\starpi_e(\tau) \cdel_{e,t+1/3}
		(\tau)^2\Bigg)
	-\underbrace{\sum_{\sigma^2}
	\f{
	\nu_{e,t+1/3}(\sigma\,|\,\bL)}
	{\starpi_e(\sigma^1)}
	(\cdel_{e,t+1/3})^2(\sigma^2\,|\,\bL)
	}_{\textup{denote this }S^1(\sigma^1\,|\,\bL)}
	\Bigg\}\,.
	\end{align*}
We then use 
\eqref{e:rough.inductive.bounds.forall.t}
and \eqref{e:orth} to estimate (with $\cdel\equiv \cdel_{e,t+1/3}$
and $\xi\equiv \xi_{e,t+1/3}$)
	\begin{align*}
	S(\sigma^1\,|\,\bL)
	&=
	\sum_{\sigma^2}
	\bigg\{
	1+\xi_{e,t+1/3}(\sigma\,|\,\bL)\bigg\}
	\starpi_e(\sigma^2)
	(\cdel_{e,t+1/3})^2(\sigma^2\,|\,\bL)\\
	&\le
	\begin{cases}\DS
	k^{O(1)}\bigg(
	\cdel\xi + \f{\cdelred\dot{\xi}}{ 2^k}\bigg)
	\le \f{k^{O(1)}}{2^{k\zeta}}
	\bigg(
	\cdel
	+	\f{\cdelred}{2^k}
		\bigg)
	& \textup{if }\sigma^1\ne\red\,,\medskip\\
	\DS
	k^{O(1)}\bigg(
	\cdel\dot{\xi} + \f{\cdelred \ddot{\xi}}{2^k}
		\bigg)
	\le \f{k^{O(1)}}{2^{k\zeta}}
		(\cdel + \cdelred)
	& \textup{if }\sigma^1=\red\,.
	\end{cases}
	\end{align*}
By substituting this into the preceding calculation
 and combining with
\eqref{e:one.third.to.one.half.rho.comparison}, 
we conclude that
	\[
	\begin{pmatrix}
	\cdel_{e,t+1/2}\\
	\cdelred_{e,t+1/2}
	\end{pmatrix}
	\le 
	\f{k^{O(1)}}{2^{k\zeta}}
	\begin{pmatrix}
	1 & 2^{-k} \\
	1 & 1
	\end{pmatrix}
	\begin{pmatrix}
	\cdel_{e,t+1/3} \\
	\cdelred_{e,t+1/3}
	\end{pmatrix}\,.
	\]
We combine this with the error incurred by going from time $t+1/2$ to $t+2/3$ --- this results from the change in the message $u$, which again is bounded by \eqref{e:ctypes.one.two.vrelerr}. Altogether we conclude
	\beq\label{e:ctypes.one.two.cond}
	\begin{pmatrix}
	\cdel_{e,t+2/3}\\
	\cdelred_{e,t+2/3}
	\end{pmatrix}
	\le
	\f{k^{O(1)}}{2^{k\zeta}}
	\begin{pmatrix}
	1& 2^{-k} \\
	1 &1
	\end{pmatrix}
	\begin{pmatrix}
	\cdel_{e,t+1/3} \\
	\cdelred_{e,t+1/3}
	\end{pmatrix}
	+	\f{k^{O(1)}
		(\bm{\cdel}_{t+1/3}+\bm{\cdelred}_{t+1/3})
	}{2^{k(1+\zeta)}}
	\begin{pmatrix}1\\1\end{pmatrix}
	\eeq
for all $e\in\delta v$.\smallskip

\noindent\bemph{\hypertarget{p:ctypes.varupdate.UPDATE.C.MGL}{Part~5b}. Errors in marginals incurred by update C.} We now estimate the change in edge marginals between times $t+2/3$ and $t+1$, resulting from update \eqref{e:clause.frac.update}. To this end, we can express
	\begin{align*}
	\bbz_{e,t+1}\nu_{e,t+1}(\sigma,\bL)
	&= \bbz_{e,t+2/3}\nu_{e,t+2/3}(\sigma,\bL)
	\f{\pi_{\DD}(\bL\,|\,\bt_e)}{\nu_{e,t+2/3}(\bL)}
	\f{u_{e,t+1}(\sigma)}{u_{e,t+2/3}(\sigma)}\\
	&=\bbz_{e,t+2/3}\nu_{e,t+2/3}(\sigma,\bL)
	\f{\pi_{\DD}(\bL\,|\,\bt_e)}{\nu_{e,t+2/3}(\bL)}
	\Bigg\{ 1 + O\bigg( 
	\f{k^{O(1)} \bm{\fdel}_{t+2/3}}{2^{k(1+\zeta)}}
		\bigg)\Bigg\}\,,
	\end{align*}
where the last estimate comes from \eqref{e:ctypes.two.three.vrelerr}. Summing over $\sigma$ gives
	\beq\label{e:ctypes.two.three.clause}
	\nu_{e,t+1}(\bL)
	= \f{\bbz_{e,t+2/3}
		\cdot \pi_{\DD}(\bL\,|\,\bt_e)}
		{\bbz_{e,t+1}}
	\Bigg\{ 1 + O\bigg( 
	\f{k^{O(1)} \bm{\fdel}_{t+2/3}}{2^{k(1+\zeta)}}
		\bigg)\Bigg\}
	=\pi_{\DD}(\bL\,|\,\bt_e)
	\Bigg\{ 1 + O\bigg( 
	\f{k^{O(1)} \bm{\fdel}_{t+2/3}}{2^{k(1+\zeta)}}
		\bigg)\Bigg\}\,,
	\eeq
where the last equality is because both 
$\pi_{\DD}(\bL\,|\,\bt_e)$
and $\nu_{e,t+1}(\bL)$ are probability measures over $\bL$. On the other hand, the conditional measures given $\bL$ change very little as a result of update \eqref{e:clause.frac.update}: it follows from the above that
	\beq\label{e:ctypes.two.three.cond}
	\nu_{e,t+1}(\sigma\,|\,\bL)
	=\nu_{e,t+2/3}(\sigma\,|\,\bL)
	\Bigg\{ 1 + O\bigg( 
	\f{k^{O(1)} \bm{\fdel}_{t+2/3}}{2^{k(1+\zeta)}}
		\bigg)\Bigg\}
	\eeq
for all $\sigma\in\set{\RYGB}^2$. 
Combining \eqref{e:ctypes.two.three.clause} and \eqref{e:ctypes.two.three.cond} gives
	\beq\label{e:update.C.single.edge}
	\begin{pmatrix}
	\fdel_{e,t+1}\\
	\cdel_{e,t+1}\\
	\cdelred_{e,t+1}\end{pmatrix}
	\le
	\begin{pmatrix}
	0\\
	\cdel_{e,t+2/3}\\
	\cdelred_{e,t+2/3}\end{pmatrix}
	+ \f{k^{O(1)}}{2^{k(1+\zeta)}} \bm{\fdel}_{t+2/3}
		\begin{pmatrix}1\\1\\1\end{pmatrix}\,.
	\eeq
Finally,
recalling
\eqref{e:intermediate.bullet.measure}, we compare the measures
	\begin{align*}
	\bbz_{e,t+1}\nu_{e,t+1}(\sigma,\bL)
	&= \psi_{e,t+1}(\sigma,\bL)
	\hq_e(\sigma,\bL)
	u_{e,t+1}(\sigma)\,,\\
	\bbz_{e\bullet,t+1}
		\nu_{e\bullet,t+1}(\sigma,\bL)
	&= \psi_{e,t+1}(\sigma,\bL)
	\hq_e(\sigma,\bL)
	u_{e,t+1/3}(\sigma)\,.
	\end{align*}
It follows by combining
\eqref{e:ctypes.one.two.vrelerr} and \eqref{e:ctypes.two.three.vrelerr} that
	\beq\label{e:update.C.single.edge.back.to.bullet}
	\begin{pmatrix}
	\fdel_{e\bullet,t+1}\\
	\cdel_{e\bullet,t+1}\\
	\cdelred_{e\bullet,t+1}
	\end{pmatrix}
	\le\begin{pmatrix}
	\fdel_{e,t+1}\\
	\cdel_{e,t+1}\\
	\cdelred_{e,t+1}\end{pmatrix}
	+\f{k^{O(1)}
	}{2^{k(1+\zeta)}}
	\bigg(\bm{\cdel}_{t+1/3}+\bm{\cdelred}_{t+1/3}
	+\bm{\fdel}_{t+2/3}\bigg)
	\begin{pmatrix}1\\1\\1\end{pmatrix}\,.
	\eeq
This concludes our analysis of update C.\smallskip

\noindent\bemph{\hypertarget{p:ctypes.varupdate.CONCLUSION}{Part~6}. Convergence of iterative procedure.} We now collect the bounds obtained above to prove that the iterative procedure described in \hyperlink{p:ctypes.varupdate.ITERATIVE.CONSTRUCTION}{Part~1} converges. Recall from \eqref{e:update.A.def.epsilon} the definition of $\vec{\ep}(t)$. For $t\ge1$, it follows by combining \eqref{e:ctypes.one.two.crelerr} (from \hyperlink{p:ctypes.varupdate.UPDATE.B.MSG}{Part~4a}) and \eqref{e:error.in.avhq.from.update.C} (from \hyperlink{p:ctypes.varupdate.UPDATE.C.MSG}{Part~4b}) that we have
	\beq\label{e:vec.eps.t.bound}
	\vec{\ep}(t)
	\le 
	\f{k^{O(1)}}{2^{k\zeta}}
	\begin{pmatrix}
	2^{-k} & 2^{-k} & 2^{-k}\\
	1 & 1 & 1\\
	2^{k\zeta}
	&2^{-k}&2^{k\zeta} 
	\end{pmatrix}
	\begin{pmatrix}
	\fdel_{e,t-1/3}\\
	\cdel_{e,t-2/3}\\
	\cdelred_{e,t-2/3}
	\end{pmatrix}\,.
	\eeq
It will be useful to abbreviate $\bm{K}\equiv\bm{\cdel}+\bm{\cdelred}$.
Recall that $\err_v(t)$ is defined analogously to
\eqref{e:defn.of.err.v.notation}, so
	\[\err_v(t)\le
	\sum_{e\in\delta v}
	\f{k^{O(1)}}{2^{k\zeta}}
	\begin{pmatrix}
	1 & 2^{-k} & 2^{-k(1+\zeta)}
	\end{pmatrix}
	\begin{pmatrix}
	2^{-k} & 2^{-k} & 2^{-k}\\
	1 & 1 & 1\\
	2^{k\zeta} & 2^{-k} &2^{k\zeta}
	\end{pmatrix}
	\begin{pmatrix}
	\fdel_{e,t-1/3}\\
	\cdel_{e,t-2/3}\\
	\cdelred_{e,t-2/3}
	\end{pmatrix}
	\le \f{k^{O(1)}
		(\bm{K}_{t-2/3}
		+\bm{\fdel}_{t-1/3}
	)}{2^{k(1+\zeta)}}
	\,.\]
We then substitute these bounds into the analysis of update A, from \hyperlink{p:ctypes.varupdate.UPDATE.A.GENERAL}{Part~3b}: from \eqref{e:fdel.bounds.from.update.A} 
and \eqref{e:cdel.bounds.from.update.A} we obtain
	\beq\label{e:update.A.single.edge}
	\begin{pmatrix}
	\fdel_{e,t+1/3}\\
	\cdel_{e,t+1/3}\\
	\cdelred_{e,t+1/3}
	\end{pmatrix}
	\le
	\begin{pmatrix}
	1&0&0\\
	1&1&0\\
	1&0&1
	\end{pmatrix}
	\begin{pmatrix}
	\fdel_{e\bullet,t}\\
	\cdel_{e\bullet,t}\\
	\cdelred_{e\bullet,t}
	\end{pmatrix}
	+\f{k^{O(1)}}{2^{k(1+\zeta)}}
	\Bigg\{
	\f{\bm{\fdel}_{t-1/3}+\bm{K}_{t-2/3}}
		{2^{k(1+\zeta)}} 
	+ \f{\cdel_{e,t-2/3}
	+\cdelred_{e,t-2/3}
	+\fdel_{e,t-1/3}}{2^{k\zeta}}
	\Bigg\}
	\begin{pmatrix}
	1\\1\\2^k
	\end{pmatrix}\,.
	\eeq
Next, from the analysis of the marginal errors resulting from update B (\hyperlink{p:ctypes.varupdate.UPDATE.B.MGL}{Part~5a}), we have
the bounds
 \eqref{e:one.to.two.thirds.clause} and
\eqref{e:ctypes.one.two.cond}, which we recall give
	\beq\label{e:update.B.single.edge}
	\begin{pmatrix}
	\fdel_{e,t+2/3}\\
	\cdel_{e,t+2/3}\\
	\cdelred_{e,t+2/3}
	\end{pmatrix}
	\le
	\f{k^{O(1)}}{2^{k\zeta}}
	\begin{pmatrix}
	2^{k\zeta} &0&0\\
	0&1& 2^{-k} \\
	0&1 &1
	\end{pmatrix}
	\begin{pmatrix}
	\fdel_{e,t+1/3} \\
	\cdel_{e,t+1/3} \\
	\cdelred_{e,t+1/3}
	\end{pmatrix}
	+	\f{k^{O(1)} \bm{K}_{t+1/3}
	}{2^{k(1+\zeta)}}
	\begin{pmatrix}1\\1\\1\end{pmatrix}\,.
	\eeq
From the analysis of marginal errors resulting from update C (\hyperlink{p:ctypes.varupdate.UPDATE.C.MGL}{Part~5b}), we have the bounds 
\eqref{e:update.C.single.edge}
and \eqref{e:update.C.single.edge.back.to.bullet},
which combine to give
	\beq\label{e:update.C.single.edge.combined.final}
	\begin{pmatrix}
	\fdel_{e\bullet,t+1}\\
	\cdel_{e\bullet,t+1}\\
	\cdelred_{e\bullet,t+1}
	\end{pmatrix}
	\le\begin{pmatrix}
	0\\
	\cdel_{e,t+2/3}\\
	\cdelred_{e,t+2/3}\end{pmatrix}
	+\f{k^{O(1)}
	}{2^{k(1+\zeta)}}
	\bigg(
	\bm{\fdel}_{t+2/3}
	+\bm{K}_{t+1/3}
	\bigg)
	\begin{pmatrix}1\\1\\1\end{pmatrix}\,.
	\eeq
Combining the above bounds gives, for all $t\ge1$,
	\begin{align*}
	\begin{pmatrix}
	\fdel_{e\bullet,t}\\
	\cdel_{e\bullet,t}\\
	\cdelred_{e\bullet,t}
	\end{pmatrix}
	&\stackrel{
	\eqref{e:update.C.single.edge.combined.final}
	}{\le}
	\begin{pmatrix}
	0\\
	\cdel_{e,t-1/3}\\
	\cdelred_{e,t-1/3}\end{pmatrix}
	+\f{k^{O(1)}
	}{2^{k(1+\zeta)}}
	\bigg(
	\bm{\fdel}_{t-1/3}
	+\bm{K}_{t-2/3}\bigg)
	\begin{pmatrix}1\\1\\1\end{pmatrix}\\
	&\stackrel{\eqref{e:update.B.single.edge}}{\le}
	\f{k^{O(1)}}{2^{k\zeta}}
	\begin{pmatrix}
	0 &0&0\\
	0&1& 2^{-k} \\
	0&1 &1
	\end{pmatrix}
	\begin{pmatrix}
	\fdel_{e,t-2/3} \\
	\cdel_{e,t-2/3} \\
	\cdelred_{e,t-2/3}
	\end{pmatrix}
	+\f{k^{O(1)}
	}{2^{k(1+\zeta)}}
	\bigg(
	\bm{\fdel}_{t-1/3}
	+\bm{K}_{t-2/3}\bigg)
	\begin{pmatrix}1\\1\\1\end{pmatrix}\,.
	\end{align*}
Substituting these bounds 
(along with \eqref{e:update.B.single.edge})
into \eqref{e:update.A.single.edge} gives
	\beq\label{e:fdel.cdel.single.edge.bound}
	\begin{pmatrix}
	\fdel_{e,t+1/3}\\
	\cdel_{e,t+1/3}\\
	\cdelred_{e,t+1/3}
	\end{pmatrix}
	\le
	\f{k^{O(1)}}{2^{k\zeta}}
	\left\{
	\begin{pmatrix} 
	0&0&0\\
	0&1&0\\
	0&1&1
	\end{pmatrix}
	\begin{pmatrix}
	\fdel_{e,t-2/3}\\
	\cdel_{e,t-2/3}\\
	\cdelred_{e,t-2/3}
	\end{pmatrix}
	+\f{\bm{\Upsilon}_{t-2/3}}{2^k}
	\begin{pmatrix}1\\1\\1
	\end{pmatrix}
	\right\}\,,
	\eeq
where $\bm{\Upsilon}\equiv\bm{\fdel}+\bm{K}$.
(In the $3\times3$ matrix on the right-hand side of \eqref{e:fdel.cdel.single.edge.bound}, some entries are zero because we absorbed the errors into the $\bm{\Upsilon}_{t-2/3}$ term.)
Aggregating the last bound over $e\in\delta v$ gives
	\[
	\bm{\Upsilon}_{t+1/3}
	\equiv
	\bm{\fdel}_{t+1/3}+\bm{K}_{t+1/3}
	\le \f{k^{O(1)}}{2^{k\zeta}}
	\Bigg\{\bm{\fdel}_{t+1/3}+\bm{K}_{t+1/3}\Bigg\}
	= \f{k^{O(1)}}{2^{k\zeta}}
	\bm{\Upsilon}_{t-2/3}\,,
	\]
for all $t\ge1$. This implies that the iterative procedure of \hyperlink{p:ctypes.varupdate.ITERATIVE.CONSTRUCTION}{Part~1} converges to the desired weights
$\Psi_v\equiv\Psi_{v,\infty}$. We now turn to proving
that the weights satisfy the bounds \eqref{e:defn.of.err.v.notation} and \eqref{e:ctypes.varupdate.MAINBOUND}. 
Summing \eqref{e:fdel.cdel.single.edge.bound} over $t\ge1$ gives
	\begin{align*}
	S_{\ge1} &\equiv
	\begin{pmatrix}
	1&0\\1&1
	\end{pmatrix}
	\sum_{t\ge1}
	\begin{pmatrix}
	\cdel_{e,t+1/3}\\
	\cdelred_{e,t+1/3}
	\end{pmatrix}
	\le
	\begin{pmatrix}
	1&0\\1&1
	\end{pmatrix}
	\sum_{t\ge1}
	\f{k^{O(1)}}{2^{k\zeta}}
	\Bigg\{
	\begin{pmatrix} 
	1&0\\
	1&1
	\end{pmatrix}
	\begin{pmatrix}
	\cdel_{e,t-2/3}\\
	\cdelred_{e,t-2/3}
	\end{pmatrix}
	+\f{\bm{\Upsilon}_{t-2/3}}{2^k}
	\begin{pmatrix}1\\1\end{pmatrix}
	\Bigg\}\\
	&\le
	\f{k^{O(1)}}{2^{k\zeta}}
	\Bigg\{
	\begin{pmatrix}
	1&0\\1&1
	\end{pmatrix}
	\sum_{t\ge1}
	\begin{pmatrix}
	\cdel_{e,1/3}\\
	\cdelred_{e,1/3}
	\end{pmatrix}
	+ \f{\bm{\Upsilon}_{1/3}}{2^k}
	\begin{pmatrix}1\\1\end{pmatrix}
	+ S_{\ge1}
	\Bigg\}\,,
	\end{align*}
and rearranging this inequality gives an upper bound on $S_{\ge1}$. It follows using
\eqref{e:fdel.cdel.single.edge.bound} again that
	\begin{align*}
	\sum_{t\ge1}\begin{pmatrix}
	\cdel_{e,t+1/3}\\
	\cdelred_{e,t+1/3}
	\end{pmatrix}
	&\le
	\begin{pmatrix}
	\cdel_{e,1/3}\\
	\cdelred_{e,1/3}
	\end{pmatrix}
	+\f{k^{O(1)}}{2^{k\zeta}}
	\Bigg\{
	\begin{pmatrix} 
	1&0\\
	1&1
	\end{pmatrix}
	\begin{pmatrix}
	\cdel_{e,1/3}\\
	\cdelred_{e,1/3}
	\end{pmatrix}
	+\f{\bm{\Upsilon}_{1/3}}{2^k}
	\begin{pmatrix}1\\1\end{pmatrix}
	+S_{\ge1}
	\Bigg\}\\
	&\le
	k^{O(1)}\Bigg\{
	\begin{pmatrix}
	1 & 0 \\ 2^{-k\zeta} & 1
	\end{pmatrix}
	\begin{pmatrix}
	\cdel_{e,1/3}\\
	\cdelred_{e,1/3}
	\end{pmatrix}
	+\f{\bm{\Upsilon}_{1/3}}{2^{k(1+\zeta)}}
		\begin{pmatrix}1\\1\end{pmatrix}
	\Bigg\}\,,
	\end{align*}
where the last step makes use of the upper bound on $S_{\ge1}$. We can also use 
\eqref{e:update.B.single.edge}
and \eqref{e:fdel.cdel.single.edge.bound}
to bound the sum of $\fdel_{e,t+2/3}$ over all $t\ge1$, so altogether we have
	\beq\label{e:sum.over.t.positive}
	\sum_{t\ge1}\begin{pmatrix}
	\fdel_{e,t+2/3}\\
	\cdel_{e,t+1/3}\\
	\cdelred_{e,t+1/3}
	\end{pmatrix}
	\le
	k^{O(1)}\Bigg\{
	\begin{pmatrix}
	0&0&0\\
	0&1&0\\
	0&2^{-k\zeta} & 1
	\end{pmatrix}
	\begin{pmatrix}
	\fdel_{e,1/3}\\
	\cdel_{e,1/3}\\
	\cdelred_{e,1/3}
	\end{pmatrix}
	+\f{\bm{\Upsilon}_{1/3}}{2^{k(1+\zeta)}}
		\begin{pmatrix} 1\\ 1\\1\end{pmatrix}
	\Bigg\}
	\eeq
We can then apply 
\eqref{e:fdel.bounds.from.update.A} and \eqref{e:cdel.bounds.from.update.A} with $t=0$ to obtain
	\[
	\begin{pmatrix}
	\fdel_{e,1/3}\\
	\cdel_{e,1/3}\\
	\cdelred_{e,1/3}
	\end{pmatrix}
	\le
	\begin{pmatrix}
	1&0&0\\
	1&1&0\\
	1&0&1
	\end{pmatrix}
	\begin{pmatrix}
	\fdel_{e\bullet,0}\\
	\cdel_{e\bullet,0}\\
	\cdelred_{e\bullet,0}
	\end{pmatrix}
	+\f{k^{O(1)}}{2^{k(1+\zeta)}}
	\Bigg\{
	 \err_v
	+\dot{\ep}_e
	+\f{\min\set{\ddot{\ep}_e,1}}
		{2^{k(1+\zeta)}}
	\Bigg\}
	\begin{pmatrix}
	1 \\ 1\\ 2^k
	\end{pmatrix}\,.
	\]
Finally, since we see from \eqref{e:table.of.all.measures}
that $\nu_{e,-2/3}
\cong \hat{p}_e \phi_e u_{e,-2/3}$
while
$\nu_{e\bullet,0}
\cong \hat{q}_e \phi_e u_{e,-2/3}$, we conclude
that 
	\[\begin{pmatrix}
	\fdel_{e\bullet,0}\\
	\cdel_{e\bullet,0}\\
	\cdelred_{e\bullet,0}
	\end{pmatrix}
	\le
	\underbrace{\begin{pmatrix}
	\fdel_{e,-2/3}\\
	\cdel_{e,-2/3}\\
	\cdelred_{e,-2/3}
	\end{pmatrix}
	}_{\textup{zero}}
	+
	\begin{pmatrix}
	1 & 2^{-k} & 2^{-k}\\
	1 & 2^{-k} & 2^{-k}\\
	1 & 1 & 1
	\end{pmatrix}
	\begin{pmatrix}
	\ep_e \\
	\dot{\ep}_e\\
	\min\set{\ddot{\ep}_e,1}/2^{k\zeta}
	\end{pmatrix}\,,
	\]
where the first term on the right-hand side vanishes because $\nu_{e,-2/3}$ is fully judicious by the assumption. Combining the last two bounds gives
	\[\begin{pmatrix}
	\fdel_{e,1/3}\\
	\cdel_{e,1/3}\\
	\cdelred_{e,1/3}
	\end{pmatrix}
	\le\begin{pmatrix}
	1 & 2^{-k} & 2^{-k}\\
	1 & 2^{-k} & 2^{-k}\\
	1 & 1 & 1
	\end{pmatrix}
	\begin{pmatrix}
	\ep_e \\
	\dot{\ep}_e\\
	\min\set{\ddot{\ep}_e,1}/2^{k\zeta}
	\end{pmatrix}
	+\f{k^{O(1)} \err_v}
		{2^{k(1+\zeta)}}
	\begin{pmatrix}
	1 \\ 1\\ 2^k
	\end{pmatrix}\,.
	\]
This implies $\bm{\Upsilon}_{1/3} \le O(\err_v)$. It follows by combining with
\eqref{e:sum.over.t.positive} that
	\beq\label{e:final.sum.bound}
	\sum_{t\in\mathbb{Z},t\ge 0}
	\begin{pmatrix}
	\fdel_{e,t+2/3}\\
	\cdel_{e,t+1/3}\\
	\cdelred_{e,t+1/3}
	\end{pmatrix}
	\le
	\begin{pmatrix}
	1 & 2^{-k} & 2^{-k}\\
	1 & 2^{-k} & 2^{-k}\\
	1 & 1 & 1
	\end{pmatrix}
		\begin{pmatrix}
	\ep_e \\
	\dot{\ep}_e\\
	\min\set{\ddot{\ep}_e,1}/2^{k\zeta}
	\end{pmatrix}
	+\f{k^{O(1)}\err_v}{2^{k(1+\zeta)}}
	\begin{pmatrix}
	1 \\ 1\\ 2^k
	\end{pmatrix}\,.
	\eeq
We remark that 
\eqref{e:final.sum.bound}
implies that the quantities $\fdel,\cdel,\cdelred$ satisfy the claimed bound 
\eqref{e:rough.inductive.bounds.forall.t}.
The quantities $\alpha,\xi$ can be bounded in terms 
of $\fdel,\cdel,\cdelred$, so it is straightforward to deduce that 
\eqref{e:rough.inductive.bounds.forall.t} indeed holds.
To conclude, it follows by recalling \eqref{e:cond.judicious.update}
and \eqref{e:defn.cdel.kappa}
that
	\[\Bigg|\f{(\psi_{e,\infty})^j
		(\tau\,|\,\bL)}
		{(\phi_e)^j(
		\tau\,|\,\bL)}-1\Bigg|
	\le
	O\Bigg(
	\sum_{t\in\mathbb{Z},t\ge 0}
	\Bigg|
	\f{(\psi_{e,t+1})^j
		(\tau\,|\,\bL)}
		{(\psi_{e,t})^j(
		\tau\,|\,\bL)}-1\Bigg|
		\Bigg)
	\le O\Bigg(
	\sum_{t\ge0}
	\cdel_{e,t+1/3}(\tau)
	\Bigg)\,.\]
It follows by recalling
\eqref{e:clause.frac.update}
and \eqref{e:fdel.defn}
that
	\[\sum_{t\in\mathbb{Z},t\ge 0}
	\Bigg|
	\f{\psi_{e,t+1}(\bL)}{\psi_{e,t}(\bL)}-1\Bigg|
	\le O(1)
	\sum_{t\ge0}
	\fdel_{e,t+2/3}\,.
	\]
These quantities are bounded by \eqref{e:final.sum.bound}, and 
this implies the claimed bound 
\eqref{e:ctypes.varupdate.MAINBOUND} for the clause-dependent weights. 
The errors in the clause-independent weights $\psi_v(\usi_{\delta v})$ can be bounded by applying
Proposition~\ref{p:var.update} with input errors $\vec{\ep}(t)$ for integers $t\ge0$, where
(recalling \eqref{e:update.A.def.epsilon}) we have
	\begin{align}\nonumber
	\vec{E}
	&\equiv \sum_{t\in\mathbb{Z},t\ge0} \vec{\ep}(t)
	=\sum_{t\in\mathbb{Z},t\ge0}
	\Bigg\{\crelerr(\avhq_{t-2/3},
	\avhq_{t-1/3})
	+\crelerr(\avhq_{t-1/3},
	\avhq_{t})
	\Bigg\}\\
	&\stackrel{\eqref{e:vec.eps.t.bound}}{\le}
	\vec{\ep}(0)
	+\f{k^{O(1)}}{2^{k\zeta}}
	\begin{pmatrix}
	2^{-k} & 2^{-k} & 2^{-k}\\
	1 & 1 & 1\\
	2^{k\zeta} & 2^{-k} & 2^{k\zeta} 
	\end{pmatrix}
	\sum_{t\in\mathbb{Z},t\ge0} 
	\begin{pmatrix}
	\fdel_{e,t+2/3}\\
	\cdel_{e,t+1/3}\\
	\cdelred_{e,t+1/3}
	\end{pmatrix} \nonumber \\
	&\stackrel{\eqref{e:final.sum.bound}}{\le}
	k^{O(1)}
	\begin{pmatrix}
	1 & 2^{-k} & 2^{-k} \\
	1 & 1 & 1 \\
	1 & 1 & 1 
	\end{pmatrix}
	\begin{pmatrix}
	\ep_e\\
	\dot{\ep}_e\\
	\min\set{\ddot{\ep}_e,1}/2^{k\zeta}
	\end{pmatrix}
	+\f{k^{O(1)}\err_v}{2^{k\zeta}}
	\begin{pmatrix}
	2^{-k}\\ 1 \\ 2^{k\zeta}
	\end{pmatrix}\,.
	\label{e:final.bound.sum.eps.over.t}
	\end{align}
It follows from this that 
	\begin{align*}
	E_e + \f{\dot{E}_e}{2^k}
	+ \f{\min\set{\ddot{E}_e,1}}{2^{k(1+\zeta)}}
	&\le k^{O(1)}
	\Bigg\{
	\ep_e + \f{\dot{\ep}_e}{2^k}
	+ \f{\min\set{\ddot{\ep}_e,1}}{2^{k(1+\zeta)}}
	 + \f{\err_v}{2^{k(1+\zeta)}}
	\Bigg\}\,,\\
	\dot{E}_e
	+\f{\min\set{\ddot{E}_e,1}}{2^{k\zeta}}
	&\le
	k^{O(1)}
	\Bigg\{
	\ep_e + \dot{\ep}_e
	+ \min\set{\ddot{\ep}_e,1}
	 + \f{\err_v}{2^{k\zeta}}
	\Bigg\}\,.
	\end{align*}
Thus, if we apply the bound Proposition~\ref{p:var.update}
with input errors $\vec{\ep}(t)$ 
summed over all $t\ge0$, the final bound
will still be of the form \eqref{e:defn.of.err.v.notation}, as claimed.\smallskip

\noindent\bemph{\hypertarget{p:ctypes.varupdate.OUTMSG}{Part~7}. Error bound for outgoing messages.}
It remains to verify \eqref{e:ctypes.varupdate.OUTMSG}. This follows by a similar argument as in \hyperlink{c:var.update.BETTER.MGL.BOUND}{Step~2 of the proof of Corollary~\ref{c:var.update}}. 
Fix an edge $e\in\delta v$, and define $\hat{s}_e = \hat{p}_e$ and $\hat{s}_e=\hq_{e'}$ for $e'\in\delta v\setminus e$. We can apply the above result to obtain weights $\Theta_v$
such that $\nu_{\delta v}[\Theta_v;\hat{s}]$
is fully judicious. On the same edge $e$, the outgoing \textsc{bp} messages for the measures
$\nu_{\delta v}[\Phi_v;\hat{p}]$,
$\nu_{\delta v}[\Theta_v;\hat{s}]$, and 
$\nu_{\delta v}[\Psi_v;\hq]$ are given by
	\begin{align*}
	\dot{p}_e(\sigma,\bL)
	&\cong \phi_e(\sigma,\bL) u_p(\sigma)\\
	\dot{s}_e(\sigma,\bL)
	&\cong \theta_e(\sigma,\bL) u_s(\sigma)\\
	\dq_e(\sigma,\bL)
	&\cong \psi_e(\sigma,\bL) u_q(\sigma)
	\end{align*}
where $u_p \equiv u_{e,-2/3}$, $u_q \equiv u_{e,\infty}$, and analogously
	\[
	u_s(\sigma_e)
	\cong \sum_{\usi_{\delta a \setminus e}}
	\varphi_v(\usi_{\delta v})
	\theta_v(\usi_{\delta v})
	\prod_{e'\in\delta a \setminus e}\Bigg\{
	\sum_{\bL}
	\theta_{e'}(\sigma_{e'},\bL_{e'})
	\hat{s}_{e'}(\sigma_{e'},\bL_{e'})
	\Bigg\}\,.
	\]
Conditional on $\bL$, we use \eqref{e:reweighted.messages.v.to.c} to define the reweighted messages 
$P_e(\sigma\,|\,\bL)$,
$S_e(\sigma\,|\,\bL)$,
$Q_e(\sigma\,|\,\bL)$.
Recall \eqref{e:delta.vec.of.t}, and note that
	\[
	\vec{D} 
	\equiv
	\vrelerr(u_p,u_q)
	\le O\Bigg(
	\sum_{t\in\mathbb{Z},t\ge0}
	\vrelerr
	(u_{e,t-2/3},u_{e,t+1/3})
	\Bigg)
	=O\Bigg( 
	\sum_{t\in\mathbb{Z},t\ge0}
	\vec{\delta}(t+1/3)\Bigg)\,.\]
Let $\vec{D}'$ denote the first three entries of $\vec{D}$. Then 
	\begin{align*}
	\vec{D}'
	&
	\stackrel{\eqref{e:vec.delta.bound.with.L}}{\le}
	k^{O(1)}
	\sum_{t\in\mathbb{Z},t\ge0}
	\left\{
	\begin{pmatrix}
	1\\1\\1
	\end{pmatrix} \err_v(t)
	+ 
	\begin{pmatrix}
	0&0&0\\
	0&1&1\\
	0&1&1
	\end{pmatrix}
	\begin{pmatrix}
	\ep_e(t)\\
	\dot{\ep}_e(t)\\
	\min\set{\ddot{\ep}_e(t),1}/2^{k\zeta}
	\end{pmatrix}
	\right\} \\
	&\stackrel{\eqref{e:final.bound.sum.eps.over.t}}{\le} k^{O(1)}\left\{
	\begin{pmatrix}
	1\\1\\1
	\end{pmatrix} \err_v
	+ 
	\begin{pmatrix}
	0&0&0\\
	0&1&1\\
	0&1&1
	\end{pmatrix}
	\begin{pmatrix}
	\ep_e\\
	\dot{\ep}_e\\
	\min\set{\ddot{\ep}_e,1}/2^{k\zeta}
	\end{pmatrix}\right\}
	\end{align*}
Meanwhile, the errors between the edge weights
$\phi_e(\sigma,\bL)$,
$\theta_e(\sigma,\bL)$, and $\psi_e(\sigma,\bL)$ are bounded by \eqref{e:ctypes.varupdate.MAINBOUND}.
From this it is easy to see that
	\[
	\max_{\bL}\Bigg\{
	\Bigg|
	\f{\dot{s}_e(\bL)}{\dot{p}_e(\bL)}-1
	\Bigg|+ 
	\Bigg|
	\f{\dq_e(\bL)}{\dot{p}_e(\bL)}-1
	\Bigg|
	+\Bigg\|
	\f{S_e(\cdot\,|\,\bL)}{P_e(\cdot\,|\,\bL)}-1
	\Bigg\|_\infty
	\Bigg\}
	\le
	k^{O(1)}
	\err_v\,.
	\]
Combining with the result of 
Lemma~\ref{l:marginal.error} gives
(cf.\ \eqref{e:marginal.error.P.to.S})
	\[
	\max_{\bL}\Bigg\{\Bigg|
	\f{(S_e)^j(\tau\,|\,\bL)}
		{(P_e)^j(\tau\,|\,\bL)}-1
	\Bigg|\Bigg\}
	\le \f{k^{O(1)}\err_v}{2^{k\zeta}}\,.
	\]
The error between $S_e$ and $Q_e$ is bounded by
(cf.\ \eqref{e:S.Q.error} and \eqref{e:marginal.error.S.to.Q})
	\begin{align*}
	\max_{\bL}\Bigg\{
	\Bigg|
	\f{Q_e(\sigma\,|\,\bL)}{S_e(\sigma\,|\,\bL)}-1
	\Bigg|\Bigg\}
	&\le 
	k^{O(1)} \Bigg(
	\err_v
	+\Ind{\red[\sigma]\ge1}
	\bigg(
	\dot{\ep}_e+
	\f{\min\set{\ddot{\ep}_e,1}}{2^{k\zeta}}
	\bigg)\Bigg\} \\
	\max_{\bL}\Bigg\{\Bigg|
	\f{(Q_e)^j(\tau\,|\,\bL)}
		{(S_e)^j(\tau\,|\,\bL)}-1
	\Bigg|\Bigg\}
	&\le k^{O(1)}\Bigg( 
	\ep_e 
	+\f1{(2^k)^{\Ind{\tau\ne\red}}}
	\bigg\{ \dot{\ep}_e
	+ \f{\min\set{\ddot{\ep}_e,1}}{2^{k\zeta}}
	\bigg\}\Bigg)\,.
	\end{align*}
Combining these bounds gives the claim \eqref{e:ctypes.varupdate.OUTMSG}.
\end{proof}
\end{ppn}

\begin{proof}[Proof of Proposition~\ref{p:contraction.for.simple.var}]
We must argue that the iterative construction of Definition~\ref{d:lagrange.defn.Augmented} converges to the desired weights $\Psi\equiv\Psi_{\DD}(U;\omega_{\delta U})$ asserted by the proposition. The argument closely follows the outline of the proofs of
Propositions~\ref{p:nice.tree.lagrange} and \ref{p:induct}. For each leaf edge $e=(au)\in\delta U$, it is easy to check that at the initial stage $t=0$ we have
	\[
	\hq_{e,0}(\sigma,\bL)
	\cong \pi_{\DD}(\bL\,|\,\bt_{p(e)})
	\prodhq_e(\sigma)\,,
	\]
where $p(e)=(av)$ (as in \eqref{e:pi.DD.first.appearance}). On the other hand, for each $e\in\delta U$ we are given the edge marginal
	\[
	\Omega_e(\sigma,\bL)
	=\Omega_e(\bL)
	\Omega_e(\sigma\,|\,\bL)
	= \pi_{\DD}(\bL\,|\,\bt_{p(e)})
	\omega_{\bL,j}(\sigma)\,,
	\]
with $j=j(\bt_e)$. Applying Step~I in Definition~\ref{d:lagrange.defn.Augmented} gives
	\[
	\dq_{e,1}(\sigma,\bL)
	\cong \f{\Omega_e(\sigma,\bL)}
		{\hq_{e,0}(\sigma,\bL)}
	\cong \f{\omega_{\bL,j}(\sigma)}
		{\prodhq_e(\sigma)}\,.
	\]
It follows that $\VRELERR(\dq_{e,0},\dq_{e,1})$
(in the notation of Definition~\ref{d:rel.error.notation.w.types})
can be bounded in the same way as
$\vrelerr(\dq_{e,0},\dq_{e,1})$
is bounded by Lemma~\ref{l:first.update.at.boundary}, using assumption~\eqref{e:apriori.assump.omega.Augmented} in place of assumption~\eqref{e:apriori.assump.omega}. We can then repeat the argument from the 
\hyperlink{p:induct.proof}{proof of Proposition~\ref{p:induct}}, with the following modifications:
\begin{enumerate}[--]
\item The base case is the bound on
$\VRELERR(\dq_{e,0},\dq_{e,1})$,
(as opposed to
$\vrelerr(\dq_{e,0},\dq_{e,1})$ in the original argument);
\item We apply 
Proposition~\ref{p:clause.update} conditional on the clause type $\bL$; and
\item We apply Proposition~\ref{p:ctypes.varupdate}
in place of Proposition~\ref{p:var.update}
and Corollary~\ref{c:var.update}.
\end{enumerate}
This modification gives that for any edge $e$ in $U$, 
the variable-to-clause message error 
$\VRELERR(\dq_{e,t-1},\dq_{e,t})$ satisfies the bound \eqref{e:induction.variable.error},
while the clause-to-variable message error 
$\CRELERR(\hq_{e,t-1},\hq_{e,t})$
 satisfies the bound \eqref{e:induction.clause.error}. Substituting these bounds into the 
\hyperlink{p:nice.tree.lagrange.proof}{proof of Proposition~\ref{p:nice.tree.lagrange}}
gives the result.
\end{proof}

\section{Solution of second moment optimization}
\label{s:merge}

In this section we complete the proof of 
the key second moment estimate Proposition~\ref{p:second.moment.judicious}, assuming an \textit{a~priori} estimate (Proposition~\ref{p:apriori} below) which will be proved in Section~\ref{s:burnin}. For an edge $e=(av)$ in the processed graph $\GG$ with $\bL_a=\bL$ and $j(e)=j\in[k]$, we will write $\omega_e\equiv \omega_{\bL,j}$. Recall that in Proposition~\ref{p:contraction.for.simple.var}, we let $U$ represent the depth-one neighborhood of a non-compound variable $v$ (see Definition~\ref{d:judicious.augmented.alphabet}), and considered the optimization problem
$\nu=\optnu(U;\omega_{\delta U})$
for $\omega$ satisfying \eqref{e:apriori.assump.omega.Augmented}.
The main result \eqref{e:depth.one.update.final}
from Proposition~\ref{p:contraction.for.simple.var}
shows that for the optimizer $\nu$, the discrepancy $\disc_{av}(\nu)$ for $a\in\pd v$ is small relative to the maximum discrepancy $\disc_e(\nu)$ over all edges $e\in\delta U$: indeed we can summarize
\eqref{e:depth.one.update.final} more simply as
	\beq\label{e:depth.one.update.final.maxversion}
	\disc_{av}(\nu)
	\le k^{O(1)}(\vth)^{1/4}
	\Bigg\{
	\max_{e\in\delta U}
	\disc_e(\omega)\Bigg\}\,.
	\eeq
(Recall from Definition~\ref{d:error.notation.edge} that $\vth\equiv 2^{-k\zeta/6}$ for an absolute constant $0<\zeta\le1/20$. Recall also that the discrepancy measure $\disc$ is defined by \eqref{e:def.discrepancy.measure.e} and \eqref{e:disc.augmented} for the compound and non-compound settings, respectively.) The main technical result of this section is the analogue of Proposition~\ref{p:contraction.for.simple.var} for compound regions:

\begin{ppn}[contraction result for compound regions]
\label{p:contraction.COMPOUND}
Assume that $R\ge k$, where $R$ is the neighborhood radius in \eqref{e:radii}.
In the same setting as 
Proposition~\ref{p:update.compound},
consider the constrained optimization problem
	\[
	\optnu(U;\omega_{\delta U})
	= \argmax_\nu
	\Bigg\{ \Ent(\nu)
		: 	\nu \in
		\Judicious(U,\omega_{\delta U})\Bigg\}\,,\]
where $U$ is a compound enclosure (Definition~\ref{d:enclosure}), and we assume that
(cf.\ \eqref{e:apriori.assump.omega.Augmented})
	\beq\label{e:apriori.assump.omega.COMPOUND}
	\max_{\sigma\in\set{\RYGB}^2}
	\Bigg\{
	\f1{(2^k)^{\Ind{\sigma=\red\red}}}
	\bigg|
		\f{\omega_e(\sigma)}
		{\prodom_e(\sigma)}
		- 1 \bigg|
		\Bigg\} \le
		\f1{2^{2k\zeta}}
	\eeq
for all $e\in\delta U$. For $\nu\equiv\optnu(U;\omega_{\delta U})$, let $\disc_e(\nu)$ be as in \eqref{e:def.discrepancy.measure.e}. For any edge $e$, define
	\beq\label{e:choice.of.C}
	C_e \equiv 
	\begin{cases}
	1 &\textup{if $e$ is nice,} \\
	\DS C(k,R) \equiv \f1{\min_{\bt}
	\min\set{
		\starpi_{\bt}(\sigma)^6
		:\sigma\in\supp\starpi_{\bt} }}
	 &\textup{if $e$ is not nice.}
	\end{cases}
	\eeq
(The choice of $C\equiv C(k,R)$ is very crude; but we emphasize that it is a finite constant independent of $n$.) Then for every edge $e$ in $U\setminus\delta U$ we have
	\beq\label{e:block.update.final}
	\f{\disc_e(\nu)}{C_e} \le \f14\Bigg\{
	\max_{e\in\delta U} \disc_e(\omega)
	\Bigg\}\,,
	\eeq
provided that the constants $\KAPPA$ and $\DELTACONST$ 
from Definitions~\ref{d:j.defective}~and~\ref{d:contained}
satisfy
	\[
	\KAPPA\ge \bigg(\f{240}{\zeta}\bigg)^4\,,\quad
	\DELTACONST \le \min\bigg\{
		\f{\zeta}{30},
		\f{\log 2}{2(\KAPPA)^{1/2}}
	\bigg\}\,.
	\]
We let $\Psi\equiv\Psi(U;\omega_{\delta U})$ denote the Lagrangian weights such that
the solution $\nu=\optnu(U;\omega_{\delta U})$ of the above coincides with $\nu[U;\Psi]$, i.e., the $\Psi$-weighted Gibbs measure on $U$ for the pair coloring model. If $\omega$ is judicious and close to $\prodom$, then $\omega$ defines a nonempty subspace $\Judicious(U;\omega_{\delta U})$ by the considerations of Lemma~\ref{l:dimension.of.judicious.omega}, so $\Psi$ exists by the general theory of Lagrange multipliers.
\end{ppn}

The proof of Proposition~\ref{p:contraction.COMPOUND} occupies most of this section, and uses the result of Proposition~\ref{p:nice.tree.lagrange} as the main inductive building block. Before turning to this, we first explain in \S\ref{ss:hess} how the results of Propositions~\ref{p:contraction.for.simple.var} and \ref{p:contraction.COMPOUND} can be combined with the \textit{a~priori} estimate (Proposition~\ref{p:apriori}) to prove Proposition~\ref{p:second.moment.judicious}. 

\subsection{From contraction to optimization}
\label{ss:hess}

The \textit{a~priori} estimate requires some conditions on the processed neighborhood profile $\DD$
(Definition~\ref{d:degseq}), which we now formally introduce.

\begin{dfn}[expansion condition]\label{d:expansion}
In the processed graph $\GG=(V,F,E)$, 
for any subset of variables $S\subseteq V$, let $ F_\bullet(S)$ denote the subset of clauses
$a\in F$ with $|S\cap \pd a|\ge 9k/10$.
We say that a subset of variables $S\subseteq V$, is a \bemph{type-subset}
if membership in $S$ can be determined by the variable type alone --- i.e., for all $v\in S$ and $w\notin S$,
we have $\bT_v \ne \bT_w$.
We say that $\GG$ \bemph{expands on type-subsets} if every type-subset $S\subseteq V$ satisfies the bound
	\beq\label{e:expansion.bd}
	|F_\bullet(S)|\le \begin{cases}
		|V| 2^{29k/30} & 
		\textup{if $|S|/|V| \le 4/5$,}\\
		|S| &\textup{if $|S|/|V| \le 1/16$}
		\end{cases}\eeq
Whether $\GG$ expands on type-subsets
can be determined from the neighborhood profile $\DD$, so we say equivalently that $\DD$ \bemph{expands on type-subsets}.
\end{dfn}

\begin{lem}[expansion result; proved in Section~\ref{s:burnin}]\label{l:expand.on.types}
Let $\GG'$ be the random $\ksat$ instance, and $\GG=\proc\GG'$ the processed graph, with neighborhood profile $\DD$. It holds with high probability that $\DD$ expands on type-subsets in the sense of Definition~\ref{d:expansion}.
\end{lem}

For the next proposition, we recall that $\bm{I}_0$ appears in Lemma~\ref{l:if.neg.def}, and denotes the subset of edge marginals $\omega$ that are judicious and consistent with $z\in I_0$, where $I_0$ is defined by \eqref{e:middle.interval}.

\begin{ppn}[\textit{a~priori} estimate;
proved in Section~\ref{s:burnin}]
\label{p:apriori}
Let $\GG'$ be the random $\ksat$ instance, and $\GG=\proc\GG'$ the processed graph, with neighborhood profile $\DD$. Recall Lemma~\ref{l:if.neg.def}, and let
	\beq\label{e:opt.omega.of.DD}
	\omega(\DD)
	\in
	\argmax_\omega
	\Bigg\{
	\bm{\Psi}_{\DD,2}(\omega)
	:\omega\in\bm{I}_0
	\Bigg\}\,.
	\eeq
(We do not yet argue that $\omega(\DD)$ is unique;
that will be shown in the \hyperlink{p:second.moment.judicious.proof}{proof of 
Proposition~\ref{p:second.moment.judicious}}.) There exists an absolute constant $\ZETA>0$ such that if $\DD$ expands on type-subsets
in the sense of Definition~\ref{d:expansion}, then
any maximizer $\omega(\DD)$ must satisfy the \textup{``a~priori''} bounds
	\beq\label{e:apriori}
	\Bigg|
		\f{\omega_{\bL,j}(\sigma)}
		{\prodom_{\bL,j}(\sigma)}
		- 1 \Bigg| \le
		\begin{cases}
		2^{-k\ZETA}& \textup{if $\red[\sigma]\le1$,}\\
		2^{k(1-\ZETA)}&\textup{if $\red[\sigma]=2$.}
		\end{cases}
	\eeq
whenever $\bL(j)$ is a strongly non-defective edge type in the sense of Remark~\ref{r:defective.clauses}. 
\end{ppn}

We conclude this subsection by explaining how Propositions~\ref{p:contraction.for.simple.var}, \ref{p:contraction.COMPOUND}, and \ref{p:apriori} combine to give the key second moment estimate, 
Proposition~\ref{p:second.moment.judicious}. We begin with an elementary lemma:

\begin{lem}[relative entropy identity]\label{l:relative.entropy}
Let $\mathcal{A}$ be any finite set, and suppose
 $\nu(a)=\Lm(a)/z$ is a probability measure over $a\in\mathcal{A}$. We then have
	\[
	\Ent(\mu\,|\,\nu)
	=\Bigg\{\Ent(\nu)- \Ent(\mu)\Bigg\}
	+\Big\langle \nu-\mu,\log\Lm\Big\rangle
	\]
for any probability measure $\mu$ over $\mathcal{A}$.

\begin{proof}
The relative entropy of $\mu$ with respect to $\nu$ is
	\[
	\Ent(\mu\,|\,\nu)
	= \bigg\langle \mu, \log\f{\mu}{\nu}\bigg\rangle
	= -\Ent(\mu) - \Big\langle\mu,\log\nu\Big\rangle
	= -\Ent(\mu) - \bigg\langle\mu,\log
		\f{\Lm}{z}\bigg\rangle
	= -\Ent(\mu) - \Big\langle\mu,\log\Lm\Big\rangle
		 + \log z\,.
	\]
Applying this identity with $\mu=\nu$ gives
	\[
	0 = \Ent(\nu\,|\,\nu)
	= -\Ent(\nu) - \Big\langle\nu,\log\Lm\Big\rangle
		 + \log z\,.
	\]
Subtracting the two expressions gives the conclusion.
\end{proof}
\end{lem}

\begin{rmk}[applications of Lemma~\ref{l:relative.entropy}]
\label{r:relative.entropy}
We now informally describe the two main ways that we will apply the identity obtained in Lemma~\ref{l:relative.entropy}.
Let $\LIN$ be an affine subspace of the simplex of probability measures over $\mathcal{A}$.

\begin{enumerate}[A.]
\item Suppose that $\nu(a)=\Lm(a)/z$ where $\log\Lm$ are the Lagrange multipliers for the subspace $\LIN$, such that $\nu$ is the solution of the optimization problem 
	\beq\label{e:constrained.opt.abstract}
	\argmax_{\nu'}
	\Bigg\{ \Ent(\nu')
	: \nu' \in \LIN\Bigg\}\,.\eeq
If $\mu$ is any element of $\LIN$, then Lemma~\ref{l:relative.entropy} gives
	\beq\label{e:relative.entropy.with.Lagrangian}
	\Ent(\mu\,|\,\nu)
	=\Bigg\{
	\underbrace{\Ent(\nu)- \Ent(\mu)}
		_{\textup{nonnegative}}\Bigg\}
	+\underbrace{
		\Big\langle \nu-\mu,\log\Lm\Big\rangle
		}_{\textup{zero}}
	=\Ent(\nu)- \Ent(\mu)\,,
	\eeq
where the last equality uses that 
$\langle\mu,\log\Lm\rangle$ must be constant over $\mu\in\LIN$, by the nature of Lagrange multipliers.
In the \hyperlink{p:second.moment.judicious.proof}{proof of Proposition~\ref{p:second.moment.judicious}}, we will lower bound $\Ent(\mu\,|\,\nu)$, and use the above identity to deduce a lower bound on $\Ent(\nu)-\Ent(\mu)$.

\item Now suppose instead that $\mu$
solves \eqref{e:constrained.opt.abstract},
while 
$\nu(a)=\Lm(a)/z$
also belongs to $\LIN$,
but no longer with the assumption that $\log\Lm$ are the Lagrange multipliers for $\LIN$.
Then Lemma~\ref{l:relative.entropy} gives
	\beq\label{e:rel.entropy.bound.by.weights}
	\Ent(\mu\,|\,\nu)
	=\Bigg\{
	\underbrace{\Ent(\nu)- \Ent(\mu)
	}_{\textup{nonpositive}}
	\Bigg\}
	+\Big\langle \nu-\mu,\log\Lm\Big\rangle
	\le
	\Big\langle \nu-\mu,\log\Lm\Big\rangle\,,
	\eeq
where the inequality holds since $\Ent(\mu)\ge\Ent(\nu)$ by assumption. 
If $\Lm$ are the Lagrange multipliers for $\LIN$, then 
$\langle\mu,\log\Lm\rangle$ must be constant over $\mu\in\LIN$, so we would have
$\langle\nu-\mu,\log\Lm\rangle=0$ and thus
$\Ent(\mu\,|\,\nu)=0$, as already noted.
However, in some of the applications that follow, we will only have $\Lm$ ``close to'' the Lagrange multipliers for $\LIN$, such that 
$\langle\mu,\log\Lm\rangle$ is
roughly constant over $\mu\in\LIN$.
This will imply that
$\langle\nu-\mu,\log\Lm\rangle$
is small, so that $\mu$ must be close to $\nu$. 
For the formal details, see the proof of
Proposition~\ref{p:defect.join}. 
\end{enumerate}
We emphasize that although Lemma~\ref{l:relative.entropy} is a trivial identity, it plays an essential part in the proof of the key results of this paper.
\end{rmk}

\begin{proof}[\hypertarget{p:second.moment.judicious.proof}{Proof of Proposition~\ref{p:second.moment.judicious}}] By the result of Lemma~\ref{l:if.neg.def}, it suffices to show that when we restrict to $\omega\in\bm{I}_0$, the function $\bm{\Psi}=\bm{\Psi}_{\DD,2}$ is uniquely maximized at $\omega(\DD)=\prodom$ (in the notation of \eqref{e:opt.omega.of.DD}), with negative-definite Hessian at the maximizer. For the proof, we assume that $\DD$ is bounded away from zero in the sense of \eqref{e:pos.frac}; this holds with high probability by Proposition~\ref{p:posfrac}. We also assume that the profile $\DD$ expands on type-subsets, which by Lemma~\ref{l:expand.on.types} occurs with high probability. Now take any maximizer $\omega=\omega(\DD)$ from \eqref{e:opt.omega.of.DD} --- we have not yet argued that it is unique. By the expansion condition and the \textit{a~priori} estimate Proposition~\ref{p:apriori} (and recalling from Definition~\ref{d:error.notation.edge} that $\zeta=\ZETA/4$), we know that $\omega$ satisfies the estimate \eqref{e:apriori} on all $(\bL,j)$ such that $\bL(j)$ is a strongly non-defective edge type in the sense of Remark~\ref{r:defective.clauses}. \smallskip

\noindent\bemph{Proof that $\omega(\DD)=\prodom$
	uniquely maximizes
	$\bm{\Psi}=\bm{\Psi}_{\DD,2}$ on $\bm{I}_0$.}
We start by letting $\omega=\omega(\DD)$ be any maximizer from
\eqref{e:opt.omega.of.DD}. Consequently, in Proposition~\ref{p:update.compound},
for any choice of compound enclosure $U$,
on the left-hand side of
\eqref{e:max.over.acute.omega.compound}
the maximum over $\acute{\omega}$ must be
$\bm{\Psi}(\acute{\omega})
-\bm{\Psi}(\omega)=0$. 
Therefore the right-hand side of
\eqref{e:max.over.acute.omega.compound} also vanishes:
	\[
	\max_{\nu'}
	\Bigg\{ \Ent(\nu')
		: 	\nu' \in
		\Judicious(U,\omega_{\delta U})
		\Bigg\}
	=\max_{\mu'}
	\Bigg\{ \Ent(\mu')
		: 	\mu' \in
		\Simplex(U;\omega_U)\Bigg\}
	\]
Since the entropy function $\Ent$ is strictly concave on 
$\Judicious(U,\omega_{\delta U})\supseteq
\Simplex(U;\omega_U)$, this implies $\mu=\nu$ where 
	\begin{align}\nonumber
	\nu
	&=\optnu(U;\omega_{\delta U})
	= \argmax_{\nu'}
	\Bigg\{ \Ent(\nu')
		: 	\nu' \in
		\Judicious(U,\omega_{\delta U})\Bigg\}\,,\\
	\mu
	&= \mu(U;\omega)
	=\argmax_{\mu'}
	\Bigg\{ \Ent(\mu')
		: 	\mu' \in
		\Simplex(U;\omega_U)\Bigg\}\,.
	\label{e:mu.nu}
	\end{align}
(Note that $\mu,\nu$ are uniquely defined by the strict concavity of $\Ent$.) 
Likewise, by Proposition~\ref{p:block.update.non.compound}, if $U$ is the depth-one neighborhood of any non-compound variable, then we must have
$\mu=\nu$ where
$\nu=\optnu(U;\omega_{\delta U})$
and $\mu=\mu(U;\omega_U)$
(the only difference being that these are now
 in the augmented pair coloring model).
We now argue that $\omega_{\bL,j}=\prodom_{\bL,j}$ in all cases where $\bL$ is nice. Indeed, let
	\beq\label{e:max.disc.over.nondefective}
	D(\omega) = \max_e
	\Bigg\{
	\disc_e(\omega)
	: \textup{$e$ is nice}\Bigg\}\,,
	\eeq
and let $\emax$ be any nice edge that attains the maximum. We have two cases:
\begin{enumerate}[1.]
\item Suppose $\emax$ is an interior edge in a compound enclosure $U$. Recall from
Definition~\ref{d:enclosure} that
the boundary of $U$ consists of variables which are perfect (Definition~\ref{d:perfect.fair}), hence also orderly. We saw in Definition~\ref{d:orderly} that an orderly variable cannot lie within distance $(\DELTACONST)^3$ of any defect, so every $e\in\delta U$ is strongly non-defective in the sense of Remark~\ref{r:defective.clauses}. Thus Proposition~\ref{p:apriori} can be applied, and guarantees that the condition
\eqref{e:apriori.assump.omega.COMPOUND}
of Proposition~\ref{p:contraction.COMPOUND} is satisfied by every $e\in\delta U$
(again, recall from Definition~\ref{d:error.notation.edge} that we take $\zeta=\ZETA/4$). Then the result
\eqref{e:block.update.final} of Proposition~\ref{p:contraction.COMPOUND} 
(with $C_{\emax}=1$, since $\emax$ is nice)
gives
that $\nu=\optnu(U;\omega_{\delta U})$ must satisfy
	\[
	\disc_{\emax}(\nu)
	\le
	\f14\Bigg\{
	\max_{e\in\delta U} \disc_e(\omega)
	\Bigg\}
	\le \f{D(\omega)}{4}\,,
	\]
where the last bound uses that all edges of $\delta U$ must be nice. On the other hand, as we noted
in \eqref{e:mu.nu} above,
$\nu$ must coincide with $\mu=\mu(U;\omega)$
by Proposition~\ref{p:update.compound}, so in fact we must have
	\[
	D(\omega) =\disc_{\emax}(\omega)=\disc_{\emax}(\nu) 
	\le \f{D(\omega)}{4}\,.
	\]
This proves $D(\omega)=0$ for the case that $e$ lies in the interior of a compound enclosure.

\item Suppose instead that $\emax=(av)$ where $v$ is a non-compound variable. It follows from Definition~\ref{d:enclosure} that $v$ must be perfect, so by the same reasoning as above it cannot lie within distance $(\DELTACONST)^3$ of any defect. 
Thus, if $U$ is the depth-one neighborhood of $v$,
every $e\in\delta U$ must be strongly non-defective
in the sense of Remark~\ref{r:defective.clauses}. It follows by Proposition~\ref{p:apriori} that the condition~\eqref{e:apriori.assump.omega.Augmented}
of Proposition~\ref{e:apriori.assump.omega.Augmented} is satisfied for every $e\in\delta U$
(again using that $\zeta=\ZETA/4$). Then, the result
\eqref{e:depth.one.update.final} of Proposition~\ref{p:contraction.for.simple.var}
(or its consequence \eqref{e:depth.one.update.final.maxversion}) implies that $\nu=\optnu(U;\omega_{\delta U})$ must satisfy
	\[
	\disc_{\emax}(\nu)
	\le k^{O(1)} (\vth)^{1/4}
	\Bigg\{\max_{e\in\delta U} \disc_e(\omega)
	\Bigg\}
	\le k^{O(1)} (\vth)^{1/4}
	D(\omega) \,,
	\]
where the last bound again uses that all edges of $\delta U$ must be nice. On the other hand, we again have that $\nu$ must coincide with $\mu=\mu(U;\omega)$ by Proposition~\ref{p:block.update.non.compound}, so in fact we obtain the bound
	\[
	D(\omega)
	=\disc_{\emax}(\omega)
	=\disc_{\emax}(\nu)
	\le \f{D(\omega)}{4}\,.
	\]
This proves $D(\omega)=0$ for the case that $e$
does not lie in the interior of a compound enclosure.
\end{enumerate}
The above proves $D(\omega)=0$, i.e., the optimizer $\omega=\omega(\DD)$ coincides with the canonical product measure $\prodom$ on all nice edges. If $e$ is a non-nice edge, then it must lie in the interior of a compound enclosure $U$. In this case, the result \eqref{e:block.update.final} of Proposition~\ref{p:contraction.COMPOUND} gives,
with $C$ as in \eqref{e:choice.of.C},
	\beq\label{e:defective.nonquant.bound}
	\disc_e(\omega)
	\le \f{C}{4}\Bigg\{
	\max_{e\in\delta U} \disc_e(\omega)
	\Bigg\}
	\le \f{C D(\omega)}{4} = 0\,,
	\eeq
again using that all edges of $\delta U$ must be nice. This proves that $\omega=\omega(\DD)$ coincides with $\prodom$ on all edges, so that
$\prodom$ is the unique maximizer of $\bm{\Psi}$ on $\bm{I}_0$ as claimed. 
\smallskip

\noindent\bemph{Proof that $\bm{\Psi}=\bm{\Psi}_{\DD,2}$ has negative-definite Hessian at $\prodom$.}
Now suppose that $\omega\in\bm{I}_0$ is close in euclidean norm to the optimizer $\omega(\DD)=\prodom$. 
 As in \eqref{e:max.disc.over.nondefective}, 
let $D(\omega)$ denote the maximal discrepancy under $\omega$ over all nice edges, and let $\emax$
be any nice edge with $\disc_{\emax}(\omega)=D(\omega)$. If $\emax$ lies in the interior of a compound enclosure then let $U$ be that enclosure. If instead $\emax=(av)$ for a non-compound variable $v$, then let $U$ be the depth-one neighborhood of $v$. In either case we let
$\nu=\optnu(U;\omega_{\delta U})$
and $\mu=\mu(U;\omega_U)$, as in \eqref{e:mu.nu}. Note that since $\Ent$ is strictly concave,
 if $\omega$ is very close to $\prodom$ then $\nu=\optnu(U;\omega_{\delta U})$ must be very close to $\prodnu\equiv\optnu(U;\prodom_{\delta U})$, and likewise $\mu=\mu(U;\omega_U)$ must be very close to $\prodmu\equiv\mu(U;\prodom_U)$. But since we proved above that $\prodom$ is the unique optimizer for $\bm{\Psi}$ on $\bm{I}_0$, we must have 
$\prodnu=\prodmu$, and therefore $\mu$ must be very close to $\nu$. Next, by Lemma~\ref{l:relative.entropy} and the first calculation \eqref{e:relative.entropy.with.Lagrangian} in Remark~\ref{r:relative.entropy}, we have
	\beq\label{e:relent.first.app}
	\Ent(\nu)-\Ent(\mu)
	=\Bigg\{\Ent(\nu)-\Ent(\mu)\Bigg\}
	+\underbrace{
	\Big\langle\log\Lm, \nu-\mu\Big\rangle
	}_{\textup{zero}}
	= \Ent(\mu\,|\,\nu)
	\ge
	\max_e
	\Ent(\mu_e\,|\,\nu_e)
	\,,\eeq
where $\Lm=\Psi(U;\omega_{\delta U})$ from Proposition~\ref{p:contraction.COMPOUND}
in the compound case,
and $\Lm=\Psi_{\DD}(U;\omega_{\delta U})$
from Proposition~\ref{p:contraction.for.simple.var}
in the non-compound case. In the compound case, for $\mu$ sufficiently close to $\nu$, we have
	\[\Ent(\mu_{\emax}\,|\,\nu_{\emax})
	=\sum_\sigma
	\mu_{\emax}(\sigma)\log\f{\mu_{\emax}(\sigma)}
		{\nu_{\emax}(\sigma)}
	\ge \f13
	\sum_{\sigma}
	\f{(\mu_{\emax}(\sigma)-\nu_{\emax}(\sigma))^2}
		{\nu_{\emax}(\sigma)}
	\ge \f13 \Big(\|\mu_{\emax}-\nu_{\emax}\|_\infty\Big)^2\,.
	\]
In the non-compound case we can instead bound
	\beq\label{e:relent.lbd.worst.edge}
	\Ent(\mu_{\emax}\,|\,\nu_{\emax})
	=\sum_{\sigma,\bL}
	\mu_{\emax}(\sigma,\bL)\log\f{\mu_{\emax}(\sigma,\bL)}
		{\nu_{\emax}(\sigma,\bL)}
	\ge \f13
	\Big(\|\mu_{\emax}-\nu_{\emax}\|_\infty\Big)^2\,.
	\eeq
We remark that $\mu_e=\omega_e$ for all edges $e$ in $U$, since $\mu=\mu(U;\omega)$. Next, in the compound case we have from \eqref{e:def.discrepancy.measure.e} 
and the triangle inequality that
	\begin{align*}
	\disc_{\emax}(\omega)
	&=\sum_\sigma
	(\vth)^{\Ind{\red[\sigma]\ge1}}
	\Bigg| \f{\omega_{\emax}(\sigma)}{\prodom_{\emax}(\sigma)}
			-1
		\Bigg|
	=\sum_\sigma
	(\vth)^{\Ind{\red[\sigma]\ge1}}
	\Bigg| \f{\mu_{\emax}(\sigma)}{\prodom_{\emax}(\sigma)}
			-1
		\Bigg|\\
	&\le
	\underbrace{\sum_\sigma
	(\vth)^{\Ind{\red[\sigma]\ge1}}
	\Bigg| \f{\nu_{\emax}(\sigma)}{\prodom_{\emax}(\sigma)}
			-1
		\Bigg|}_{\disc_{\emax}(\nu)}
	+\underbrace{\sum_\sigma
	(\vth)^{\Ind{\red[\sigma]\ge1}}
	\Bigg|
	\f{\mu_{\emax}(\sigma)}{\prodom_{\emax}(\sigma)}-
	 \f{\nu_{\emax}(\sigma)}{\prodom_{\emax}(\sigma)}
		\Bigg|}_{\le
			{C_1(k)}
			\|\mu_{\emax}-\nu_{\emax}\|_\infty}\,,
	\end{align*}
where $C_1(k)$ is a constant depending only on $k$ (this uses the $\emax$ is a nice edge). Rearranging gives
	\[
	\f1{C_1(k)}
	\Bigg\{\disc_{\emax}(\omega)
	-\disc_{\emax}(\nu)
	\Bigg\}
	\le \|\mu_{\emax}-\nu_{\emax}\|_\infty\,.\]
Similarly, in the non-compound case we have from
\eqref{e:disc.augmented} and the triangle inequality that
	\beq\label{e:disc.triangle.ineq}
	\f1{C_1}
	\Bigg\{\disc_{\emax}(\omega)
	-\disc_{\emax}(\nu)
	\Bigg\}
	\le \|\mu_{\emax}-\nu_{\emax}\|_\infty
	\eeq
where $C_1$ depends only on $k$ and on the constant $c_1$ from Proposition~\ref{p:posfrac}.
By adjusting $C_1$ we can assume it is larger than the $C_1(k)$ from the preceding bound, so that 
\eqref{e:disc.triangle.ineq} holds for $\emax$ in both the compound and non-compound cases. Next, by the contraction results from Propositions~\ref{p:contraction.for.simple.var}
and \ref{p:contraction.COMPOUND}, we have
	\beq\label{e:apply.contraction}
	\disc_{\emax}(\nu)
	\le
	\f14
	\max_e\Bigg\{
	 \disc_e(\omega)
	 : \textup{$e$ is nice}\Bigg\}
	= \f14 \disc_{\emax}(\omega)\,,
	\eeq
where the last equality holds because $\emax$ was chosen to maximize $\disc_e(\omega)$ over all nice edges. (Observe that \eqref{e:apply.contraction} crucially uses the restriction to nice edges;
we recall from \eqref{e:block.update.final} that we have an extra factor $C$ (from \eqref{e:choice.of.C})
for non-nice edges.) It follows that
	\beq\label{e:main.result.of.contraction}
	\|\mu_{\emax}-\nu_{\emax}\|_\infty
	\stackrel{\eqref{e:disc.triangle.ineq}}{\ge} 
	\f1{C_1}
	\Bigg\{\disc_{\emax}(\omega)
	-\disc_{\emax}(\nu)
	\Bigg\}
	\stackrel{\eqref{e:apply.contraction}}{\ge} 
	\f3{4 C_1} \disc_{\emax}(\omega)
	=\f{3D(\omega)}{4 C_1} 
	\,.
	\eeq
Next, let $e_{\max}$ be the edge $e$ with maximal discrepancy $\disc_e(\omega)=D_{\max}(\omega)$ (this is over all edges, both nice and non-nice). If $e_{\max}$ is nice, then $D_{\max}(\omega)=D(\omega)$. 
More generally, if $D_{\max}(\omega) \le C^3 D(\omega)$ for $C$ as in \eqref{e:choice.of.C}, then the preceding bound \eqref{e:main.result.of.contraction} readily implies
	\beq\label{e:main.result.of.contraction.nice}
	\|\mu_{\emax}-\nu_{\emax}\|_\infty
	\ge \f{3D(\omega)}{4 C_1} 
	\ge \f{3D_{\max}(\omega)}{4 C_1 C^3}\,. 
	\eeq
Now suppose instead that $D_{\max}(\omega) \ge C^3 D(\omega)$. Then
	\[
	\|\nu_{e_{\max}}-\prodom_{e_{\max}}\|_\infty
	\le \f{\disc_{e_{\max}}(\nu)}{(\vth)^2}
	\stackrel{\eqref{e:block.update.final}}
		{\le}
	\f{C D(\omega)}{4(\vth)^2}
	\stackrel{\odot}{\le}
	\f{\disc_{e_{\max}}(\omega)}{4(\vth)^2 C^2}
	\le \f{\|\omega_{e_{\max}}-\prodom_{e_{\max}}
		\|_\infty}{4(\vth)^2 C}\,.
	\]
(In the above, the bound marked $\odot$ uses 
 the assumed lower bound on $D_{\max}(\omega)$.
 The last bound uses the definition \eqref{e:choice.of.C} of $C$.)
Since $(\vth)^2 C$ is a large constant, we can conclude for this case that
	\beq\label{e:main.result.of.contraction.bad}
	\|\mu_{e_{\max}}-\nu_{e_{\max}}\|_\infty
	\ge
	\|\omega_{e_{\max}}-\prodom_{e_{\max}}\|_\infty
	-\|\nu_{e_{\max}}-\prodom_{e_{\max}}\|_\infty
	\ge
	\f{\|\omega_{e_{\max}}-\prodom_{e_{\max}}\|_\infty}
		{2}
	\ge \f{D_{\max}(\omega)}{C}\,.
	\eeq
Combining \eqref{e:main.result.of.contraction.nice} and \eqref{e:main.result.of.contraction.bad} gives,
for some constant $C_2=C_2(k,R)$, the lower bound
	\beq\label{e:main.result.of.contraction.final}
	\max_e\|\mu_e-\nu_e\|_\infty
	\ge \f{D_{\max}(\omega)}{C_2}\,.
	\eeq
Since we assumed that each clause type $\bL$ occurs at least $c_1$ fraction of times, the total number of indices $(\bL,j)$ is at most $k/c_1$, so the euclidean distance between $\omega$ and $\prodom$ can be upper bounded as 
	\beq\label{e:euclidean.to.max}
	\Big(\|\omega-\prodom\|_2\Big)^2
	= \sum_{\bL,j}
		\Big(\|\omega_{\bL,j}-\prodom_{\bL,j}\|_2\Big)^2
	\le
	\f{C_3 k}{c_1}
	\Bigg( \max_e \disc_e(\omega)\Bigg)^2
	\le\f{C_3 k}{c_1} D_{\max}(\omega)^2
	\,.\eeq
Combining the bounds obtained so far gives 
	\begin{align*}
	\Ent(\nu)-\Ent(\mu)
	&\stackrel{\eqref{e:relent.first.app}}{\ge}
	\max_e \Ent(\mu_e\,|\,\nu_e)
	\stackrel{\eqref{e:relent.lbd.worst.edge}}{\ge}
	\max_e\f{(\|\mu_e-\nu_e\|_\infty)^2}{3}
	\stackrel{\eqref{e:main.result.of.contraction.final}}{\ge}
	\f13\bigg(
		\f{D_{\max}(\omega)}{C_3}
		\bigg)^2
	\\
	&\stackrel{\eqref{e:euclidean.to.max}}{\ge}
	\f13 \f{c_1(\|\omega-\prodom\|_2)^2}
	{C_3 k (C_2)^2 }
	\equiv \f1{C_4}
	\Big(\|\omega-\prodom\|_2\Big)^2\,.
	\end{align*}
To conclude, let $\omega_\nu\in\bm{I}_0$ be the edge empirical marginals of $\nu=\optnu(U;\omega_{\delta U})$. Since we already showed that $\prodom$ optimizes $\bm{\Psi}$ over $\bm{I}_0$, we must have
$\bm{\Psi}(\prodom) \ge \bm{\Psi}(\omega_\nu)$.
Combining with the results of Propositions~\ref{p:update.compound} and \ref{p:block.update.non.compound} gives
	\[
	\bm{\Psi}(\prodom)-\bm{\Psi}(\omega)
	\ge \bm{\Psi}(\omega_\nu)-\bm{\Psi}(\omega)
	\ge c_1
	\Bigg\{ \Ent(\nu)-\Ent(\mu)
	\Bigg\}
	\ge \f{c_1}{C_4}
	\Big(\|\omega-\prodom\|_2\Big)^2
	\,.
	\]
This proves that $\bm{\Psi}$ has negative-definite Hessian at $\prodom$, as claimed. As noted above, the result follows by applying Lemma~\ref{l:if.neg.def}.
\end{proof}

\subsection{Decomposition of compound enclosures}
\label{ss:breakup.of.compound.enclosures}
We now turn to the proof of Proposition~\ref{p:contraction.COMPOUND}.

\begin{dfn}[decomposition of a compound enclosure]
\label{d:joins}
Let $U$ be a compound enclosure: recall that this means $U=U^\circ\cup\pd_\circ U$ where $\pd_\circ U$ is the boundary of perfect variables. Suppose that an internal variable $v_\star\in U^\circ$ is designated as the ``root.'' Let $\KAPPA$ be the defect buffer depth from Definition~\ref{d:j.defective}, and let $\cM \equiv (\KAPPA)^{1/4}$. If a variable $u\in U^\circ$ lies at distance $\ell\cM^2$ from $v_\star$ for some positive integer $\ell$, and moreover does not lie within distance $5\cM^2$ of any non-nice variable, then we call $u$ a \bemph{terminal}. Let $\cX$ denote the set of all terminals in $U$. Let $\mathcal{A}(\cX)$ denote the set of all clauses $a$ in $U$ such that $a$ is the parent of a terminal variable (again, regarding $U$ as being rooted at $v_\star$). 
\begin{enumerate}[--]
\item If $J=J^\circ\cup\pd J^\circ$ where $J^\circ$ is the maximal connected component of $U\setminus\mathcal{A}(\cX)$ containing $v_\star$, and $\pd J^\circ\equiv\pd_\circ J$ is the set of variables in $U\setminus J^\circ$ at unit distance from $J^\circ$, then we call $J$ 
the \bemph{root join}.
\item If instead $J=J^\circ \cup(av) \cup (\pd(J^\circ\setminus v))$ where $J^\circ$ is a maximal connected component of $U\setminus\mathcal{A}(\cX)$ not containing $v_\star$, $v$ is the terminal variable in $J^\circ$ that lies closest to $v_\star$, and $a\in\mathcal{A}(\cX)$ is the parent clause of $v$, then we call $J$ a \bemph{non-root join}. We call $(av)$ the root edge of $J$.
\end{enumerate}
With a mild abuse of notation, we use $J$ to denote both the subset of variables and the induced subgraph of $U$. The join $J$ is termed \bemph{non-defective} if it does not intersect any defect. Otherwise we say that the join $J$ is \bemph{defective}. For any choice of $v_\star\in U^\circ$, this definition decomposes $U$ into joins. The joins naturally have a tree structure, where the root of the tree is the root join.
\end{dfn}

\begin{dfn}[recursive merging of joins]
\label{d:merge.joins} To prove Proposition~\ref{p:contraction.COMPOUND}, we will merge the joins recursively upwards. At each merge step, we fix an \bemph{uppermost join} $J$ which will connect the subtrees below. The root of $J$ is either 
$\vrt=v_\star$, or it is an edge $(\crt\vrt)$ where $\crt\in\mathcal{A}(\cX)$ and $\vrt\in\cX$. A leaf variable $u\in \Leaves J$ can lie arbitrarily close to $\vrt$, but the set $\cX\cap \Leaves J$ of leaves that are terminal variables lie at distance at least $\cM^2$ from $\vrt$. We let $\delta'J$ denote the edges in $J$ that are incident to $\cX\cap \Leaves J$. For each $e\in\delta'J$, we let $T_e$ be the subtree of $U$ descended from $e$. We then let $T$ be the merged tree formed from the union of $J$ with all the $T_e$ ($e\in\delta'J$). In the merge step,
 our goal will be to bound edge discrepancies (defined by \eqref{e:def.discrepancy.measure.e})
for the measure $\mu=\optnu(T;\omega_{\delta T})$
in terms of edge discrepancies for the measures $\nu(e)=\optnu(T_e;\omega_{\delta T_e})$.
\end{dfn}

\begin{rmk}\label{r:boundaries.of.joins}
Let $\Leaves U$ denote the set of all leaf variables of the compound enclosure $U$. The subset $\pd_\circ U\subseteq \Leaves U$ is the boundary of perfect variables, so these variables are not leaves in the full graph $\GG$. The variables in the complement $\Leaves U\setminus U$ are leaves in the full graph $\GG$ --- in particular, this means they are not nice (Definition~\ref{d:nice}), hence they are defective. 
We thus have two cases:
\begin{enumerate}[--]
\item If $J$ is a non-defective join, then all its leaf variables must belong either to $\pd_\circ U$ or to the terminal set $\cX$. In this case we let $\delta J$ denote the set of edges incident to all leaf variables of $J$.
\item If $J$ is a defective join, then it cannot intersect $\pd_\circ U$: this is because any defect has at its boundary a buffer of nice variables of depth at least $\KAPPA$ (Definition~\ref{d:j.defective}), so the $\KAPPA$-neighborhood of any variable in $\pd_\circ U$ is nice.
In this case we let $\delta J\equiv \delta' J$.
\end{enumerate}
These two cases will be treated separately in the analysis below.
\end{rmk}

\begin{ppn}[merge through a defective join]
\label{p:defect.join}
In the setting of Proposition~\ref{p:contraction.COMPOUND} and Definition~\ref{d:merge.joins}, consider a merge step where the uppermost join $J$ is defective. As in Definition~\ref{d:merge.joins},
let $T_e$ denote the subtrees descended from $e\in\delta J$, and $\nu(e)\equiv\optnu(T_e;\omega_{\delta T_e})$. Let $T$ be the merged tree, and
$\mu\equiv\optnu(T;\omega_{\delta T})$. Then for every edge $e'$ in $J\setminus\delta J$ we have
	\[
	\disc_{e'}(\mu)
	\le
	C_{e'} \cdot 2^{5k}
	\sum_{e\in\delta J}
	\disc_e(\nu(e))
	\]
for $C_{e'}$ as defined by \eqref{e:choice.of.C}.
\end{ppn}

\begin{ppn}[merge through a non-defective join]
\label{p:nondefect.join}
In the setting of Proposition~\ref{p:contraction.COMPOUND} and Definition~\ref{d:merge.joins}, consider a merge step where the uppermost join $J$ is non-defective.
 As in Definition~\ref{d:merge.joins},
let $T_e$ denote the subtrees descended from $e\in\delta'J$, and $\nu(e)\equiv\optnu(T_e;\omega_{\delta T_e})$. Let $T$ be the merged tree, and
$\mu\equiv\optnu(T;\omega_{\delta T})$. Then
	\[
	\sum_{e\in\delta'J}
		\disc_e(\mu)
	\le 4^{k\cM}
	\sum_{e\in\delta'J}
		\disc_e(\nu(e))\,,
	\]
where we recall that $\cM=(\KAPPA)^{1/4}$,
and the result holds provided
$\KAPPA \ge (240/\zeta)^4$ for $\zeta$ as in Definition~\ref{d:error.notation.edge}.\footnote{In fact we will see from the proof that
for Proposition~\ref{p:nondefect.join} it suffices here to have $\KAPPA > (108/\zeta)^4$, see \eqref{e:weaker.lbd.KAPPA}. We put the stronger condition
$\KAPPA \ge (240/\zeta)^4$ since this was already required for Proposition~\ref{p:contraction.COMPOUND}.}
\end{ppn}

Proposition~\ref{p:defect.join} is proved in \S\ref{ss:merge.defective},
while Proposition~\ref{p:nondefect.join} is proved in \S\ref{ss:merge.non.defect}.
We now explain how they combine to give 
Proposition~\ref{p:contraction.COMPOUND}:

\begin{proof}[\hypertarget{p:contraction.COMPOUND.proof}{{Proof of Proposition~\ref{p:contraction.COMPOUND}}}] Fix any edge $e_\star=(av)$ in $U\setminus\delta U$, and decompose $U$ as in Definition~\ref{d:joins} around $v_\star=v$. As in Proposition~\ref{p:nondefect.join}, consider a merge step where the uppermost join $J$ is non-defective. Let $T$ be the merged tree, and $\mu=\optnu(T;\omega_{\delta T})$. If $J$ contains $v_\star$, then $T=U$ and $\mu=\optnu(U;\omega_{\delta U})$. If $J$ does not contain $v_\star$, then $T$ is rooted at an edge $\ert=(\crt\vrt)$, and indeed 
	\[
	\mu
	=\optnu(T;\omega_{\delta T})
	=\optnu(T_{\ert};\omega_{\delta T_{\ert}})
	=\nu(\ert)\,,
	\]
where $\nu(\ert)$ is defined analogously to the measure $\nu(e)$ from Definition~\ref{d:merge.joins}. 
In either case, the marginal of $\mu$ on $J$ must be given by $\mu=\optnu(J;\mu_{\delta J})$. It follows by 
Proposition~\ref{p:nice.tree.lagrange} that
for any edge $(au)$ in $J$ we have
	\beq\label{e:contraction.in.good.J}
	\disc_{au}(\mu)
	\le
	k^{O(1)}
	\sum_{e\in\delta J}
		\bigg(
		\f{(\vth)^{1/4}}{2^k}
		\bigg)^{\BTW_J(e,a)}
		\disc_e(\mu)\,.
	\eeq
(To be pedantic, \eqref{e:contraction.in.good.J} follows directly from Proposition~\ref{p:nice.tree.lagrange} in the case that $J$ is rooted at an edge $\ert$. In the final merge step where $J$ is rooted at the variable $v_\star$, \eqref{e:contraction.in.good.J} follows from the variable-rooted analogue of Proposition~\ref{p:nice.tree.lagrange} --- this can be obtained by an extremely similar proof, so we will not elaborate on it here.) Recall that $\delta J$ denotes the set of edges in $J$ incident to all the leaf variables of $J$, while $\delta'J$ is the subset that are incident to terminal variables. If $e\in\delta J\setminus\delta'J$, then $\mu_e=\omega_e$. On the other hand, the contribution to the above sum from $e\in\delta'J$ is bounded by
Proposition~\ref{p:nondefect.join}. It follows that
	\begin{align}\nonumber
	\disc_{au}(\mu)
	&\le k^{O(1)}
	\sum_{e\in\delta J\setminus\delta'J}
		\bigg(
		\f{(\vth)^{1/4}}{2^k}
		\bigg)^{\BTW_J(e,a)}
		\disc_e(\omega)
	+k^{O(1)} 4^{k\cM}
		\bigg(
		\f{(\vth)^{1/4}}{2^k}
		\bigg)^{\cM^2}
	\sum_{e\in\delta'J}
		\disc_e(\nu(e))\\
	&\le
	k^{O(1)}
	\sum_{e\in\delta J\setminus\delta'J}
		\bigg(
		\f{(\vth)^{1/5}}{2^k}
		\bigg)^{\BTW_J(e,a)}
		\disc_e(\omega)
	+k^{O(1)} 
		\bigg(
		\f{(\vth)^{1/5}}{2^k}
		\bigg)^{\cM^2}
	\sum_{e\in\delta'J}
		\disc_e(\nu(e))\,,
	\label{e:recursive.disc.bound.good.case}
	\end{align}
where the last bound holds for $\cM$ large enough: to be precise, we require
	\[
	4^{k\cM}
	\le\bigg( \f{(\vth)^{1/5}}{(\vth)^{1/4}}\bigg)^{\cM^2}
	=\bigg( \f1{(\vth)^{1/20}}\bigg)^{\cM^2}
	= 2^{k\zeta \cM^2/120}\,,
	\]
where the last equality uses that $\vth\equiv 2^{-k\zeta/6}$ from Definition~\ref{d:error.notation.edge}. Thus it suffices here to have 
	\[\KAPPA
	\equiv \cM^4
	\ge \bigg( \f{240}{\zeta}\bigg)^4\,.
	\]
Next note for $e\in\delta J$, the discrepancy $\disc_e(\nu(e))$ can be bounded either by
a recursive application of \eqref{e:recursive.disc.bound.good.case} if $e$ is the root of another non-defective join, or by Proposition~\ref{p:defect.join} if $e$ is the root of a defective join. Altogether it gives that for
the edge $e_\star=(av)$ where $v=v_\star$ (and with $C_{e_\star}$ as in \eqref{e:choice.of.C}), 
the maximal-entropy measure
$\mu=\optnu(U;\omega_{\delta U})$
for the full enclosure satisfies 
	\[
	\f{\disc_{e_\star}(\mu)}{C_{e_\star}}
	\le
	\sum_{e\in\delta U}
		\bigg(
		\f{(\vth)^{1/5}}{2^k}
		\bigg)^{\BTW_U(e,a)}
		\Bigg\{
		2^{5k} \bigg(
		\f{2^k}{(\vth)^{1/5}}
		\bigg)^{\cM^2}
		\Bigg\}^{\mathfrak{B}(v_\star,e)}
		\disc_e(\omega)\,,
	\]
where $\mathfrak{B}(v_\star,e)$ is the number of defective variables between $v_\star$ and $e$, and is a crude upper bound on the number of defective joins intersecting the path between $v_\star$ and $e$. The above can be upper bounded as
	\begin{align}\nonumber
	\f{\disc_{e_\star}(\mu)}{C_{e_\star}}
	&\le
	\Bigg(\max_{e\in\delta U} \disc_e(\omega)
	\Bigg)
	\sum_{u\in\pd_\circ U}
		\f{
		(2^{2k\cM^2})^{\mathfrak{B}(v_\star,u)}
		}
		{(2^k (\vth)^{-1/5})^{d(v_\star,u)}}\\
	&\le 
	\Bigg(\max_{e\in\delta U} \disc_e(\omega)
	\Bigg)
	\sum_{u\in\pd_\circ U}
	\f{ \exp\{ k(\DELTACONST)^{-1}
		\mathfrak{B}(v_\star,v) \} }
	{ \exp\{ (k\log2)
		(1+\DELTACONST) d(v_\star,v) \}}\,,
	\label{e:constrain.kappa.delta}
	\end{align}
where for the last inequality
in \eqref{e:constrain.kappa.delta} to hold we make the following choices. First, in the numerator 
we can ensure that 
$2^{2\cM^2}$ is upper bounded by
$\exp(1/\DELTACONST)$ by requiring
	\[
	\f{\log2}{2} \ge \cM^2 \DELTACONST 
	= (\KAPPA)^{1/2} \DELTACONST \,.\]
In the denominator we can ensure that 
$(\vth)^{-1/5} = 2^{k\zeta/30}$ upper bounds $2^{k\DELTACONST}$ by requiring
	\[
	\DELTACONST \le \f{\zeta}{30}\,.
	\]
These choices together guarantee the last bound in \eqref{e:constrain.kappa.delta}. It follows by combining with the 
definition of a compound enclosure (cf.~\eqref{eq-def-rad}) that
	\[\f{\disc_{e_\star}(\mu)}{C_{e_\star}}
	\le \f14
	\Bigg(\max_{e\in\delta U}
		\disc_e(\omega)
	\Bigg)\,.\]
This proves
the result.\end{proof}

\subsection{Merging through defective joins}
\label{ss:merge.defective}


In this subsection
we prove Proposition~\ref{p:defect.join}.
The idea of the proof is to
reweight the boundary of a defective join
in such a way as to circumvent difficulty of analyzing tree recursions within a defect.


\begin{proof}[Proof of Proposition~\ref{p:defect.join}]
Recall from Definition~\ref{d:merge.joins}
the root of $J$ is either $\vrt=v_\star$, 
or it is an edge $(\crt\vrt)$.
In the first case we let $\Lmstar_J$ be as given by
Corollary~\ref{c:clause.bp.weights}. In the second case we let $\Lmstar_J$ be given by
\eqref{e:Lmstar.T.short.defn} from Remark~\ref{r:adjusted.bp.stability},
i.e., we give weight $\hqstar_{\crt\vrt}$ to the root clause $\crt$. Note that $\Lmstar_J$ includes a weight $\dqbul_e$ for each $e\in\delta J$. Therefore, for a single-copy coloring $\uta_{J\setminus\delta J}$ on $J\setminus\delta J$, we can let
	\[
	\Testar_{J\setminus\delta J}
		(\uta_{J\setminus\delta J})
	\equiv \Lmstar_J(\uta_J)
		\prod_{e\in\delta J}
	\f1{\dqbul_e(\tau_e)}\,,
	\]
where this is well-defined because the right-hand side does not depend on $\uta_{\delta J}$. We then let 
$\prodTe\equiv\Testar\otimes\Testar$.

Next, recall from Remark~\ref{r:boundaries.of.joins} that $\delta J=\delta'J$ is the set of boundary edges of $J$. Let $A$ be the subset of clauses in $J$ that are incident to $\delta J$, so
$A$ is a subset of the set $\mathcal{A}(\cX)$
from Definition~\ref{d:joins}. Let $e=(au)\in\delta J$, so $a\in A$ and $u$ is a terminal variable. Recall that $T_e$ is the subtree of $U$ descended from $e$. Let $\Lm_{T_e}=\Lm(T_e;\omega_{\delta T_e})$ be the Lagrangian weights on $T_e$ such that
$\nu(e)=\optnu(T_e,\omega_{\delta T_e})$
agrees with $\nu[T_e;\Lm_{T_e}]$. We assume moreover that they are parametrized as in Definition~\ref{d:param.weights.Lambda}, meaning that the root clause $a(e)$ of $T_e$ receives weight $\prodhq_e$ under $\Lm_{T_e}$. Write $q\equiv\BPq(T_e;\Lm_{T_e})$ for the associated \textsc{bp} messages, so in particular the downward message on $e$ is
 $\hq_e=\prodhq_e$. It will follow by induction that
 the upward message
 $\dq_e$ satisfies the bounds \eqref{e:first.update.at.boundary}
for each edge $e\in \delta J$.\smallskip

\noindent
\textbf{Step 1. Weights on merged tree.}
We now apply Proposition~\ref{p:pair.clause.weights}
to each clause $a\in A$ where we take incoming messages
$\dq_{ua}$ from below as just discussed,
but the canonical product message $\dq_{va}=\proddq_{va}$ from the parent variable $v$.
Let $h$ denote the resulting outgoing messages from the unweighted clause $a$, that is,
 $h=\BP_a[\dq]$.
Proposition~\ref{p:pair.clause.weights}
gives a clause weighting
$\Gamma_a\equiv (\gamma_e)_{e\in\delta a}$ such that for these input messages to $a$, the $\Gamma_a$-weighted clause \textsc{bp} recursion \bemph{outputs the canonical product message} $\prodhq$ on every edge in $\delta a$. For a pair configuration $\usi\equiv\usi_T$ on the merged tree $T$, we define the weight
	\[
	\Theta(\usi)
	\equiv
	\prodTe_{J\setminus\delta J}
		(\usi_{J\setminus\delta J})
	\Bigg\{
	\prod_{a\in A} \Gamma_a(\usi_{\delta a})
	\Bigg\}
	\Bigg\{
	\prod_{e\in \delta J}
	\f{\Lm_{T_e}(\usi_{T_e})}{\prodhq_e(\sigma_e) }
	\Bigg\}\,.
	\]
Recall from above that the clause $a(e)$ receives weight $\prodhq_e$ under $\Lm_{T_e}$, and in the above expression we divided by $\hq_e(\sigma_e)$ to remove this weight. Likewise, $\prodLm_J$ includes a weight $\proddq_e$ for each $e\in\delta J$, and we also divided by $\proddq_e(\sigma_e)$ to remove this weight. Let $\nu\equiv\nu[T;\Theta]$.
Let $e(a)$ denote the parent edge of $a$, i.e., the unique element of $\delta a\setminus\delta J$.
 The marginal of $\nu$ on $J\setminus\delta J$ is
	\begin{align}\nonumber
	\nu_{J\setminus\delta J}(\usi_{J\setminus\delta J})
	&\cong
	\prodTe_{J\setminus\delta J}
		(\usi_{J\setminus\delta J})
	\Bigg\{
	\sum_{\usi_{\delta J}}
	\prod_{a\in A}\Bigg(
	\Gamma_a(\usi_{\delta a})
	\prod_{e\in\delta a\setminus e(a)}
	\bigg[
	\sum_{\usi_{T_e\setminus J}}
	\f{\Lm_{T_e}(\usi_{T_e})}{\prodhq_e(\sigma_e)}
		\bigg]
	\Bigg)\Bigg\}\\ \nonumber
	&\stackrel{\textit{(a)}}{\cong}
	\prodTe_{J\setminus\delta J}
		(\usi_{J\setminus\delta J})
	\Bigg\{
	\prod_{a\in A}\Bigg(
	\sum_{\usi_{\delta a \setminus e(a)}}
	\Gamma_a(\usi_{\delta a})
	\prod_{e\in\delta a\setminus e(a)} \dq_e(\sigma_e)
	\Bigg)\Bigg\}\\
	&\stackrel{\textit{(b)}}{\cong}
	\prodTe_{J\setminus\delta J}
		(\usi_{J\setminus\delta J})
	\Bigg\{ \prod_{e'\in\delta A\setminus\delta J}
	\prodhq_{e'}(\sigma_{e'})
	\Bigg\}
	\cong
	\prodnu_{J\setminus\delta J}
	(\usi_{J\setminus\delta J})\,,
	\label{e:Theta.marginal.A}
	\end{align}
where $\prodnu=\nustar\otimes\nustar$ is the canonical product measure. (In the above calculation, the step marked \textit{(a)} uses that under the $\Lm_{T_e}$-weighted measure on $T_e$, the marginal on edge $e$ is proportional to $\prodhq_e\dq_e$. The step marked \textit{(b)} uses that the $\Gm_a$-weighted \textsc{bp} recursion outputs the canonical product message $\prodhq$ on the edge $e(a)$.) Similarly, for any $e=(bu)\in\delta J$,
 the marginal of $\nu$ on $T_e$ is given by
	\begin{align}\nonumber
	\nu_{T_e}(\usi_{T_e})
	&\cong
	\f{\Lm_{T_e}(\usi_{T_e})}{\prodhq_e(\sigma_e) }
	\sum_{\usi_{J\setminus\delta J}}
	\prodTe_{J\setminus\delta J}
		(\usi_{J\setminus\delta J})
	\Bigg\{
	\sum_{\usi_{\delta b	\setminus\set{e,e(b)}}}
	\Gm_b(\usi_{\delta b})
	\prod_{e'\in\delta b\setminus\set{e,e(b)}}
	\bigg[
	\sum_{\usi_{T_{e'} \setminus J}}
	\f{\Lm_{T_{e'}}(\usi_{T_{e'}})}{\prodhq_{e'}(\sigma_{e'}) }
	\bigg]
	\Bigg\}\\ \nonumber
	&\qquad\times\Bigg\{
	\prod_{a\in A\setminus b}
	\Bigg(
	\sum_{\usi_{\delta a\setminus e(a)}}
		\Gamma_a(\usi_{\delta a})
		\prod_{e''\in\delta a\setminus e(a)}
		\bigg[
		\sum_{\usi_{T_{e''} \setminus J}}
		\f{\Lm_{T_{e''}}(\usi_{T_{e''}})}
			{\prodhq_{e''}(\sigma_{e''}) }
		\bigg]
	\Bigg)\Bigg\}\\ \nonumber
	&\cong
	\f{\Lm_{T_e}(\usi_{T_e})}{\prodhq_e(\sigma_e) }
	\sum_{\usi_{J\setminus\delta J}}
	\prodTe_{J\setminus\delta J}
		(\usi_{J\setminus\delta J})
	\Bigg\{
	\sum_{\usi_{\delta b	\setminus \set{e,e(b)}}}
	\Gm_b(\usi_{\delta b})
	\prod_{e'\in\delta b
		\setminus \set{e,e(b)}}
	\dq_{e'}(\sigma_{e'})
	\Bigg\}
	\prod_{e''\in\delta A\setminus(\delta J\cup\delta b)}
	\prodhq_{e''}(\sigma_{e''})\\
	&\cong
	\f{\Lm_{T_e}(\usi_{T_e})}{\prodhq_e(\sigma_e) }
	\sum_{\usi_{\delta b	\setminus e}}
	\Gm_b(\usi_{\delta b})
	\Bigg\{
	\proddq_{e(b)}(\sigma_{e'})
	\prod_{e'\in\delta b
		\setminus \set{e,e(b)}}
	\dq_{e'}(\sigma_{e'})
	\Bigg\}
	\cong \Lm_{T_e}(\usi_{T_e})\,,
	\label{e:Theta.marginal.B}
	\end{align}
where the last step uses that the $\Gm_a$-weighted \textsc{bp} recursion outputs the canonical
product message $\prodhq$ on the edge $e$. This shows that the marginal of $\nu$ on $T_e$ is simply $\nu(e)=\optnu(T_e;\omega_{\delta T_e})$.
It follows that
$\nu\in\Judicious(T;\omega_{\delta T})$.
Consequently, $\mu=\optnu(T;\omega_{\delta T})$ is the entropy maximizer over
$\Judicious(T;\omega_{\delta T})$, then
$\Ent(\mu)\ge\Ent(\nu)$.
It follows by recalling Lemma~\ref{l:relative.entropy}
and Remark~\ref{r:relative.entropy} 
(in particular, the calculation
\eqref{e:rel.entropy.bound.by.weights}) that
	\beq\label{e:entropy.calculation}
	\Ent(\mu\,|\,\nu)
	= \Bigg\{ \Ent(\nu)-\Ent(\mu)\Bigg\}
	+\Big\langle \nu-\mu,
		\log\Theta\Big\rangle
	\le
	\Big\langle \nu-\mu,
		\log\Theta\Big\rangle.
	\eeq
Recall that the weights in $\Theta$ are of product form, except on the boundary edges $\delta T$
and on the edges $\delta A$. Since $\mu$ and $\nu$ agree on the subtrees $T_e$, and have the same single-copy marginals (given by $\starpi$) on all edges of the merged tree $T$, the only contribution to the right-hand side of 
\eqref{e:entropy.calculation} comes from the 
edges
 $e'\in\delta A\setminus\delta J$:
	\begin{align}\nonumber
	\Ent(\mu\,|\,\nu)
	&\le\adjustlimits
	\sum_{e'\in\delta A \setminus\delta J}
	\sum_\sigma
	\Bigg|
		\bigg\{
		\nu_{e'}(\sigma)
		-\mu_{e'}(\sigma)\bigg\}
		\log \gamma_{e'}(\sigma)
		\Bigg|\\
	&\le\adjustlimits
	\sum_{e'\in\delta A \setminus\delta J}
	\|\log\gamma_{e'}\|_\infty
	\|\mu_{e'}-\prodom_{e'}\|_\infty\,,
	\label{e:entropy.ubd.by.contribution.from.bdy}
	\end{align}
where the last step uses that 
$\nu_{e'}=\prodom_{e'}$ for all $e'\in J\setminus \delta J$. \smallskip

\noindent\bemph{Step 2. Bound on clause weights.}
We first bound the $\log\gamma_{e'}$ term in \eqref{e:entropy.ubd.by.contribution.from.bdy}.
For each $a\in A$, there is a single edge $e'=e(a)$ in $\delta a\setminus \delta J$. A crude application of Proposition~\ref{p:pair.clause.weights} gives
	\beq\label{e:log.gamma.crude.bound.eps}
	\|\log\gamma_{e(a)}\|_\infty
	\le
	k^{O(1)}
	\sum_{e''\in\delta a \setminus e(a)}
	\bigg(
	\ep_{e''}+\dot{\ep}_{e''}+\ddot{\ep}_{e''}
	\bigg)\,,
	\eeq
where $\vec{\ep}=(\ep_{e''},\dot{\ep}_{e''},\ddot{\ep}_{e''})
=\crelerr(\prodhq_{e''},h_{e''})$
 is the error between 
the unweighted messages $h=\BP_a[\dq]$
and the weighted messages $\prodhq=\BP_a[\dq;\Gamma_a]=\BP_a[\proddq]$. A crude application of Proposition~\ref{p:clause.update} 
gives, for $e''\in\delta a\setminus e(a)$,
	\beq\label{e:log.gamma.crude.bound.delta}
	\ep_{e''}+\dot{\ep}_{e''}+\ddot{\ep}_{e''}
	\le
	k^{O(1)}
	\sum_{e\in\delta a\setminus e(a)}
	\bigg(
	\delta_e
	+\dot{\delta}_e
	+\ddot{\delta}_e
	+\mdel_e
	+\mdelred_e
	\bigg)\,,
	\eeq
where $\vec{\delta}=(\delta_e,\dot{\delta}_e,\ddot{\delta}_e,\mdel_e,\mdelred_e)=\vrelerr(\proddq_e,\dq_e)$.
To bound the error between $\dq_e$ and $\proddq_e$ for $e\in\delta J$, note
	\[
	\dq_e(\sigma)
	\cong \f{\nu_e(\sigma)}{\prodhq_e(\sigma)}
	\stackrel{\eqref{e:def.discrepancy.measure.e}}{=}
	\f{\prodom_e(\sigma)}{\prodhq_e(\sigma)}
	\Bigg\{
	1 + O\bigg(\f{\disc_e(\nu)}{\vth}\bigg)
	\Bigg\}
	\cong
	\proddq_e(\sigma)
	\Bigg\{
	1 + O\bigg(\f{\disc_e(\nu)}{\vth}\bigg)
	\Bigg\}\,,
	\]
so each entry of $\vec{\delta}$ is
$O((\vth)^{-1} \disc_e(\nu))$.
Combining the last few bounds gives (again very crudely)
	\beq\label{e:apply.cor.clause.weights}
	\|\log\gamma_{e(a)}\|_\infty
	\le
	2^k 
	\sum_{e\in\delta a \setminus e(a)}
	\disc_e(\nu)
	=2^k 
	\sum_{e\in\delta a \setminus e(a)}
	\disc_e(\nu(e))
	\,,
	\eeq
for each edge $e'\in\delta A\setminus\delta J$.
\smallskip

\noindent\bemph{Step 3. Conclusion.} 
To conclude we note that for each $e\in\delta A$, we have
	\[
	\|\mu_e-\prodom_e\|_\infty
	\le
	\|\prodom_e\|_\infty
	\Bigg\|\f{\mu_e}{\prodom_e}-1\Bigg\|_\infty
	\le 
	\Bigg\|\f{\mu_e}{\prodom_e}-1\Bigg\|_\infty
	\stackrel{\eqref{e:def.discrepancy.measure.e}}{\le}
	2^k\disc_e(\mu)\,.
	\]
Combining with \eqref{e:apply.cor.clause.weights}
and substituting into \eqref{e:entropy.ubd.by.contribution.from.bdy} gives
	\begin{align}\nonumber
	\Ent(\mu\,|\,\nu)
	&\le 2^{2k} 
		\sum_{a\in A}
		\Bigg(\sum_{e\in\delta a
			\setminus e(a)}
		\disc_e(\nu(e))
		\Bigg)
		\Bigg(
		\sum_{e'\in \delta a \setminus \delta J}
			\disc_{e'}(\mu)
		\Bigg)\\
	&\le 2^{2k}
	\Bigg(\max_{e'\in J\setminus \delta J}
	\f{\disc_{e'}(\mu)}{C_{e'}}
	\Bigg)\Bigg(
	\sum_{e\in\delta J} \disc_e(\nu(e))
	\Bigg)
	\label{e:entropy.bd.defective}
	\end{align}
where $C_{e'}$ is as in the statement of the proposition (and the last bound holds simply because $C_{e'}=1$ for all $e'\in\delta A$, since those edges must be nice). On the other hand, for any $e'\in J\setminus\delta J$, we have
	\[
	\Ent(\mu\,|\,\nu)
	\ge \Ent(\mu_{e'}\,|\,\nu_{e'})
	= \Ent(\mu_{e'}\,|\,\prodom_{e'})
	\ge
	\f1{2^{3k}}
	\bigg(\f{\disc_{e'}(\mu)}{C_{e'}}\bigg)^2\,,
	\]
where we take $C_{e'}$ as in \eqref{e:choice.of.C} to account for non-nice edges where 
the minimum of $\starpi_{e'}$ on its support may be small. Combining the last two displays gives
	\[
	\max_{e'\in J\setminus \delta J}
	\f{\disc_{e'}(\mu)}{C_{e'}}
	\le 2^{5k}
	\sum_{e\in\delta J} \disc_e(\nu(e))\,,
	\]
as claimed. 
\end{proof}

\subsection{Merging
through non-defective joins}
\label{ss:merge.non.defect}

We conclude the section with the proof of
Proposition~\ref{p:nondefect.join}, showing that for a non-defective join, the total boundary discrepancy under $\mu=\optnu(T;\omega_{\delta T})$ 
(the optimizer on the merged tree)
is not too much larger than the total boundary discrepancy under the measures
$\nu(e)=\optnu(T_e;\omega_{\delta T_e})$ 
(the optimizer on the subtrees $T_e$, for $e\in\delta'J$).

\begin{proof}[Proof of Proposition~\ref{p:nondefect.join}] Recall that $\mu\equiv\optnu(T;\omega_{\delta T})$, while $\nu(e)=\optnu(T_e;\omega_{\delta T_e})$ for $e\in\delta'J$.\smallskip

\noindent\bemph{Step 1. Identification of small subtree with large discrepancy.} 
For the sake of contradiction, let us suppose the desired bound false, i.e., that we have
	\beq\label{e:nondefective.join.false.bound}
	\sum_{e\in\delta'J}
		\disc_e(\mu)
	> 4^{k\cM}
	\sum_{e\in\delta'J}
		\disc_e(\nu(e))\,.
	\eeq
Recall from Definition~\ref{d:merge.joins} that 
each variable in $\cX\cap\Leaves J$ lies at distance exactly $\cM^2$ from the root of $J$, and
$\delta'J$ denotes the edges in $J$ that are incident to $\cX\cap\Leaves J$. Similarly to \eqref{e:clause.xi}, for any clause $a$ in $J$ we define
	\[
	\xi_a(J;\mu_{\delta'J})
	\equiv
	\sum_{e\in\delta'J}
	\bigg(\f{(\vth)^{1/4}}{2^k}
		\bigg)^{\BTW_J(e,a)}
		\disc_e(\mu)\,.
	\]
Now let $A$ be the subset of clauses in $J$ at distance exactly $\cM$ from $\delta'J$. For any $e\in\delta'J$ and any $j\ge0$, we have
	\[
	\Bigg|\bigg\{a\in A:
	\BTW_J(e,a)=\cM+j\bigg\}\Bigg|
	\le (dk)^{\lceil j/2 \rceil }\,.
	\]
It follows from this that
	\begin{align}\nonumber
	\sum_{a\in A} \xi_a(J;\mu_{\delta'J})
	&=\sum_{e\in\delta'J}\disc_e(\mu)
	\sum_{a\in A}
	\bigg(\f{(\vth)^{1/4}}{2^k}\bigg)^{\BTW_T(e,a)}
	\le
	\sum_{e\in\delta'J}\disc_e(\mu)
	\sum_{j=0}^\infty
	\bigg(\f{(\vth)^{1/4}}{2^k}\bigg)^{\cM+j}
	(dk)^{\lceil j/2\rceil}\\ \nonumber
	&\le
	\Bigg\{
	\bigg(\f{(\vth)^{1/4}}{2^k}\bigg)^{\cM}
	\sum_{e\in\delta'J}\disc_e(\mu)\Bigg\}
	\Bigg(
	1 +2 \sum_{\ell=0}^\infty
		(dk)^\ell
		\bigg(
		\f{(\vth)^{1/4}}{2^k}\bigg)^{2\ell-1}
	\Bigg)\\
	&\le \bigg(\f{k(\vth)^{1/4}}{2^k}\bigg)^{\cM}
	\sum_{e\in\delta'J}\disc_e(\mu)\,.
	\label{e:nondefective.join.xi.bound}
	\end{align}
For $a\in A$, write $J_a$ for the subtree of $J$ descended from $a$. Write $\delta J_a\equiv J_a\cap\delta'J$. Combining \eqref{e:nondefective.join.false.bound} and \eqref{e:nondefective.join.xi.bound} gives
	\[
	\sum_{a\in A}
	\Bigg\{
	2\sum_{e\in\delta J_a}\disc_e(\mu)
	- \bigg( \f{2^k}{k(\vth)^{1/4}}\bigg)^{\cM}
		\xi_a(J;\mu_{\delta'J})
	- 4^{k\cM} \sum_{e\in\delta J_a}
		\disc_e(\nu(e))
	\Bigg\}\ge0\,,
	\]
so we can find some $a\in A$ for which
the expression in braces is non-negative. 
This implies two bounds: first,
	\beq\label{e:contradiction.a1}
	\xi_a(J;\mu_{\delta'J})
	\le 2 \cdot \bigg( \f{k(\vth)^{1/4}}{2^k}
	\bigg)^{\cM}
		\sum_{e\in\delta J_a} \disc_e(\mu)
	\le (k^4 (\vth)^{1/4})^\cM
		\Bigg\{
		\max_{e\in\delta'J_a} 	
		\disc_e(\mu)\Bigg\}\,,\eeq
where the last bound uses that
$|\delta J_a| \le (k^3 2^k)^{\cM}$. Secondly, we must also have
	\beq\label{e:contradiction.a2}
	\max_{e\in\delta J_a}
		\disc_e(\nu(e))
	\le\sum_{e\in\delta J_a}
		\disc_e(\nu(e))
	\le \f{2}{4^{k\cM}}
		\sum_{e\in\delta J_a}\disc_e(\mu)
	\le \bigg( \f{k^4}{2^k}\bigg)^{\cM}
	\Bigg\{\max_{e\in\delta J_a}
		\disc_e(\mu)\Bigg\}\,.
	\eeq
In the remainder of the proof we derive a contradiction.\smallskip

\noindent
\bemph{Step 2. Weights on merged tree.}
First, let $\Lm=\Lm_T=\Lm(T;\omega_{\delta T})$
denote the Lagrangian weights such that $\nu[T;\Lm]$ coincides with $\mu=\optnu(T;\omega_{\delta T})$.
We write $q$ for the corresponding \textsc{bp} messages. For the distinguished clause $a$, apply
Proposition~\ref{p:pair.clause.weights}
to obtain a clause weighting
$\Gm_a\equiv (\gm_e)_{e\in\delta a}$
such that
	\begin{align*}
	\hq_{av}
	&=\BP_{av}\bigg[
	\Big(\proddq_{ua} : u\in\pd a\setminus v\Big)
	;\Gm_a
	\bigg]\quad
	\textup{for the parent $v$ of $a$,}\\
	\prodhq_{au}
	&=\BP_{au}\bigg[
	\Big(\dq_{av}, 
	(\proddq_{u'a} : u'
		\in\pd a\setminus\set{u,v} )
	\Big)
	;\Gm_a
	\bigg]
	\quad
	\textup{for each child $u$ of $a$.}
	\end{align*}
Next, for each $e\in\delta J$,
recall that $\Lm_{T_e}\equiv \Lm(T_e;\omega_{\delta T_e})$ denotes the Lagrangian weights such that $\nu[T_e;\Lm_{T_e}]$ coincides with $\nu(e)=\optnu(T_e;\omega_{\delta T_e})$. We write $\tilde{q}$ for the corresponding \textsc{bp} messages. Let $A_a$ denote the clauses in $J_a$ incident to $\delta J_a$. For each clause $b\in A_a$,
apply Proposition~\ref{p:pair.clause.weights} again
to obtain a clause weighting $\Gm_b=(\gm_e)_{e\in\delta b}$
such that 
	\begin{align*}
	\prodhq_{bw}
	&=\BP_{bw}\bigg[
		\Big( \tilde{q}_{ua} : u\in\pd b\setminus w
		\Big)\bigg]
	\quad
	\textup{for the parent variable $w$ of $b$,}\\
	\prodhq_{bu}
	&= \BP_{bu}\bigg[
	\Big( \proddq_{wb},
		( \tilde{q}_{u'b} : u'\in\pd b\setminus
			\set{u,w})
		 \Big)
	\bigg]
	\quad
	\textup{for each child $u$ of $b$.}
	\end{align*}
Let $T_a$ denote the subtree of $T$ descended from $a$. Recall that $\Lm_T=\Lm(T;\omega_{\delta T})$
is a product of variable factors. Let
$\Lm_{T,T_a}(\usi_{T_a})$ denote the product
of those factors over the variables in $T_a$ only, and define
	\[
	\Lm_{T,T\setminus T_a}(\usi_{T\setminus T_a})
	\equiv 
	\f{\Lm_T(\usi)}{\Lm_{T,T_a}(\usi_{T_a})}\,.
	\]
For a pair configuration $\usi\equiv\usi_T$ on the merged tree $T$, we define the weight
	\[
	\Theta(\usi)
	\equiv
	\Bigg\{
	\Lm_{T,T\setminus T_a}(\usi_{T\setminus T_a})
	\Gm_a(\usi_{\delta a})
	\f{\prodTe_{J_a\setminus\delta J_a}
		(\usi_{J_a\setminus\delta J_a})}
		{\prodhq_{av}(\sigma_{av}) }
	\Bigg\}
	\Bigg\{
	\prod_{b\in A_a}\Gm_b(\usi_{\delta b})
	\Bigg\}
	\Bigg\{\prod_{e\in\delta J_a}
	\f{\Lm_{T_e}(\usi_{T_e})}
		{\prodhq_e(\sigma_e)}\Bigg\}\,.
	\]
Let $\tilde{\mu}=\nu[T;\Theta]$. It follows by similar calculations as
\eqref{e:Theta.marginal.A}
and \eqref{e:Theta.marginal.B}
that $\tilde{\mu}$ satisfies the following:
\begin{enumerate}[(i)]
\item Its marginal on $T\setminus T_a$
	agrees with that of $\mu
			= \optnu(T;\omega_{\delta T})$;
\item Its marginal on $J_a$
	agrees with that of $\prodnu
		=\optnu(T;\prodom_{\delta T})$;
\item 
	Its marginal on $T_e$
	(for each $e\in\delta J_a$)
	agrees with that of
	$\nu(e)
		= \optnu(T_e;\omega_{\delta T_e})$.
\end{enumerate}
In particular, $\tilde{\mu}\in\Judicious(T;\omega_{\delta T})$.
Recalling Lemma~\ref{l:relative.entropy}
and Remark~\ref{r:relative.entropy}
(in particular the calculation \eqref{e:rel.entropy.bound.by.weights}), we have
	\[\Ent(\mu\,|\,\tilde{\mu})
	=\Bigg\{
	\underbrace{
	\Ent(\tilde{\mu})- \Ent(\mu)
	}_{\textup{nonpositive}}
	\Bigg\}
	+\Big\langle \tilde{\mu}-\mu,\log\Theta\Big\rangle
	\le
	\Big\langle \tilde{\mu}-\mu,
	\log\Theta\Big\rangle\,.\]
Since $\mu$ and $\tilde{\mu}$ are both judicious, and agree on $T\setminus T_a$, we have
	\beq\label{e:nondefective.merge.A.B} 
	\Ent(\mu\,|\,\tilde{\mu})
	\le \Big\langle \tilde{\mu}-\mu,
	\log\Theta\Big\rangle 
	\le\Bigg\{
	\underbrace{\sum_{e\in\delta a \cap J_a}
	\Big\langle
	\tilde{\mu}_e-\mu_e, \log\gamma_e\Big\rangle
	}_{\textup{denote this $\mathcal{A}$}}\Bigg\}
	+\Bigg\{\underbrace{
	\sum_{b\in A_a}
	\sum_{e\in \delta b}
	\Big\langle
	\tilde{\mu}_e-\mu_e, \log\gamma_e\Big\rangle
	}_{\textup{denote this $\mathcal{B}$}}
	\Bigg\}\,.
	\eeq
We now turn to bounding these quantities in terms of
$\disc(\mu)$ and $\disc(\nu)$.\smallskip

\noindent
\bemph{Step 3. Bounds on entropy and clause weights.}
Let $e$ be the edge in $J_a$ with maximal $\disc_e(\mu)$. Note
	\[\bigg( \Ent(\mu\,|\,\tilde{\mu}) \bigg)^{1/2}
	\ge
	\bigg( \Ent(\mu_e\,|\,\tilde{\mu}_e)
		\bigg)^{1/2}
	\ge
	\f{\|\mu_e-\tilde{\mu}_e\|_\infty}
		{O(2^k)}
	\ge
	\f{
	\|\mu_e-\prodom_e\|_\infty
	-\|\tilde{\mu}_e-\prodom_e\|_\infty
	}{O(2^k)}\,.
	\]
If $e$ belongs to $J_a\setminus\delta J_a$, then
$\tilde{\mu}_e=\prodom_e$. If $e\in\delta J_a$, then
$\tilde{\mu}_e=\nu(e)_e$, and it follows using \eqref{e:def.discrepancy.measure.e}
 and \eqref{e:contradiction.a2} that
	\[
	\|\tilde{\mu}_e-\prodom_e\|_\infty
	\stackrel{\eqref{e:def.discrepancy.measure.e}}
		{\le}
	\f{\disc_e(\tilde{\mu})}{\vth}
	=\f{\disc_e(\nu(e))}{\vth}
	\stackrel{\eqref{e:contradiction.a2}}{\le}
	\bigg( \f{k^4}{2^k}\bigg)^{\cM}
	\f{\disc_e(\mu)}{\vth}\,.\]
This is negligible in comparison to
$\|\mu_e-\prodom_e\|_\infty$, which is lower bounded by $ \Omega(4^{-k})\disc_e(\mu)$. Therefore
	\beq\label{e:nondefective.merge.entropy.disc.lbd}
	\Ent(\mu\,|\,\tilde{\mu})
	\ge
	\f{1}{2^{6k}}\Bigg(
	\max_{e\in J_a}\disc_e(\mu)
	\Bigg)^2\,.\eeq
We next turn to bounding the clause weights.
Recall that we use $q$ to denote the \textsc{bp} messages for the $\Lm$-weighted model.
 Similarly as in
\eqref{e:log.gamma.crude.bound.eps} and \eqref{e:log.gamma.crude.bound.delta},
a crude application of
Proposition~\ref{p:pair.clause.weights} gives
	\[
	\sum_{e\in\delta a}
	\|\log\gm_e\|_\infty
	\le k^{O(1)}
	\sum_{e\in\delta a}
	\bigg(
	\delta_e
	+\dot{\delta}_e
	+\ddot{\delta}_e
	+\mdel_e
	+\mdelred_e
	\bigg)
	\]
where $\vec{\delta}=(\delta_e,\dot{\delta}_e,\ddot{\delta}_e,\mdel_e,\mdelred_e)=\vrelerr(\proddq_e,\dq_e)$. In order to bound $\vec{\delta}$ we can argue as follows. Let 
$\hat{p}_e \cong \prodom_e/\dq_e$ for all $e\in\delta a$, and let $Z_a=(\zeta_e)_{e\in\delta a}$ be the reweighting of $a$ such that
$\hat{p}_{au} = \BP_{au}[(\dq_{u'a}:u'\in\pd a\setminus u) ; Z_a]$ for all $u\in\pd a$. It follows
by another crude application of
Proposition~\ref{p:pair.clause.weights} that,
similarly as in 
\eqref{e:log.gamma.crude.bound.eps},
	\[
	\sum_{e\in\delta a}
	\|\log\zeta_e\|_\infty
	\le k^{O(1)}
	\sum_{e\in\delta a}
	(\varepsilon_e
	+\dot{\varepsilon}_e+\ddot{\varepsilon}_e
	)
	\]
where 
$\vec{\varepsilon}=(\varepsilon_e,\dot{\varepsilon}_e,\ddot{\varepsilon}_e)
=\crelerr(\hat{p}_e,\hq_e)$. We can bound $\vec{\ep}$ by noting that 
	\[
	\f{\hat{p}_e}{\hq_e}
	\cong \f{\prodom_e/\dq_e}{\mu_e/\dq_e}
	\cong \f{\prodom_e}{\mu_e}
	= 1 + O\bigg( \f{\disc_e(\mu)}{\vth}\bigg)\,,
	\]
and substituting into the previous expression gives the bound
	\[
	\sum_{e\in\delta a}
	\|\log\zeta_e\|_\infty
	\le k^{O(1)}
	 \f{\disc_e(\mu)}{\vth}\,.
	\]
From the definition of $Z_a$, for each $u\in\pd a$ we have
	\[
	\prodom_{au}
	\cong \dq_{ua} \hat{p}_{au}
	= \dq_{ua} 
	\BP_{au}\bigg[
	\Big(\dq_{u'a}:u'\in\pd a\setminus u
		\Big) ; Z_a\bigg]
	\cong
	g_{ua} \BP_{au}\bigg[ \Big(g_{u'a}
		:u'\in\pd a\setminus u\Big)\bigg]\,,
	\]
where $g_e$ is the probability measure on $\set{\RYGB}^2$ such that $g_e \cong \zeta_e\dq_e$. It follows that the measure
	\beq\label{e:nu.g.pair.model}
	\hat{\nu}_a(\usi_{\delta a})
	\cong
	\hat{\varphi}_a(\usi_{\delta a})
	\prod_{e\in\delta a} g_e(\sigma_e)
	\eeq
has marginal $\prodom_e$ on each $e\in\delta a$. This can only occur if $g_e=\proddq_e$ for all $e\in\delta a$ --- this is because $(\log g_e)_{e\in\delta a}$ must be the Lagrange multipliers for the constrained optimization problem
	\[
	\max \Bigg\{
	\Ent(\hat{\nu}):
	\textup{$\hat{\nu}$ has marginals $\prodom_e$}
	\Bigg\}\,,
	\]
and the Lagrange multipliers are unique so we must have $g_e=\proddq_e$ for all $e\in\delta a$. It follows that 
	\[\zeta_e\cong 
	\f{g_e}{\dq_e}
	\cong \f{\proddq_e }{\dq_e}
	\,,\] so the above bound on $\|\log\zeta\|_\infty$ implies a bound on $\vec{\delta}$. Altogether we obtain 
	\[
	\sum_{e\in\delta a}
	\|\log\zeta_e\|_\infty
	\le \f{k^{O(1)}}{\vth}
	\sum_{e\in\delta a} \disc_e(\mu)\,.
	\]
Substituting into the definition of $\mathcal{A}$ from
\eqref{e:nondefective.merge.A.B} gives
	\[
	\mathcal{A}
	\le 
	\f{k^{O(1)}}{(\vth)^2}
	\bigg(\sum_{e\in\delta a} \disc_e(\mu)\bigg)^2
	\le
	\f{k^{O(1)}}{(\vth)^2}
	\xi_a(J;\mu_{\delta'J})^2\,,
	\]
where the last step follows by Proposition~\ref{p:nice.tree.lagrange}.
Combining with \eqref{e:contradiction.a1} gives
	\[
	\mathcal{A}
	\le
	\f{k^{O(1)}}{(\vth)^2}
		\xi_a(J;\mu_{\delta'J})^2
	\le
	\f{k^{O(1)}(k^4 (\vth)^{1/4})^{2\cM}}{(\vth)^2} 
	\Bigg(
		\max_{e\in\delta J_a} 	
		\disc_e(\mu)\Bigg)^2\,.
	\]
Next, very similarly to the derivation of 
\eqref{e:apply.cor.clause.weights}, for $b\in A_a$ and $e\in\delta b$ we have
	\[
	\|\log\gamma_e\|_\infty
	\le 2^k \sum_{e\in\delta b
		\cap \delta J_a} \disc_e(\nu(e))\,.
	\]
(Indeed, a clause $b\in A_a$ takes in the canonical product message $\proddq$ from above, and takes from below the messages $\dq_e$ from the $\Lm_{T_e}$-weighted measures (which are precisely the measures $\nu(e)$. Thus $\|\log\gamma_e\|_\infty$ can be bounded in terms of the errors between $\dq_e$ and $\proddq_e$ for $e\in\delta b\cap \delta J_a$. These in turn can be bounded in terms of the marginal discrepancies $\disc_e(\nu(e))$, similarly as in
\eqref{e:apply.cor.clause.weights}.) Substituting into the definition of $\mathcal{B}$ from \eqref{e:nondefective.merge.A.B} gives
	\begin{align*}
	\mathcal{B}
	&\le
	2^{2k}
	\sum_{b\in A_a}
	\Bigg(
	\sum_{e\in \delta b \cap \delta J_a}
	\disc_e(\nu(e))\Bigg)
	\Bigg(\sum_{e\in\delta b}\disc_e(\mu) 
	+
	\sum_{e\in \delta b \cap \delta J_a}
	\disc_e(\nu(e))\Bigg)\\
	&\stackrel{\eqref{e:contradiction.a2}}{\le}
	2^{2k} 2k
	\Bigg(
	\max_{e'\in J_a} \disc_{e'}(\mu)
	\Bigg)
	\sum_{e\in\delta J_a}
	\disc_e(\nu(e))
	\stackrel{\eqref{e:contradiction.a2}}{\le}
	2^{2k} 2k
	\bigg( \f{k^4}{2^k}\bigg)^{\cM}
	\Bigg(\max_{e\in J_a}
		\disc_e(\mu)\Bigg)^2
	\,.
	\end{align*}
Substituting these bounds back into \eqref{e:nondefective.merge.A.B} and combining with
\eqref{e:nondefective.merge.entropy.disc.lbd} gives
	\[
	\f1{2^{6k}}
	\bigg(\max_{e\in J_a}\disc_e(\mu)\bigg)^2
	\le
	\Ent(\mu\,|\,\tilde{\mu} )
	\le \mathcal{A}+\mathcal{B}
	\le (\vth)^{\cM/3}
	\Bigg(\max_{e\in J_a}\disc_e(\mu)\Bigg)^2\,.
	\]
Again recall from Definition~\ref{d:contained} that $\vth\equiv 2^{-k\zeta/6}$. Thus, as long as we have
	\beq\label{e:weaker.lbd.KAPPA}
	\KAPPA
	=\cM^4 > \bigg(\f{108}{\zeta}\bigg)^4\,,
	\eeq
we obtain the required contradiction.
\end{proof}

\section{A priori estimates for edge marginals}\label{s:burnin}

In this section we prove Lemma~\ref{l:expand.on.types} and Proposition~\ref{p:apriori}, which were used in the \hyperlink{p:second.moment.judicious.proof}{proof of Proposition~\ref{p:second.moment.judicious}}. Lemma~\ref{l:expand.on.types} is a fairly easy expansion result, whose proof appears at the start of \S\ref{ss:conclusion.apriori} below. Proposition~\ref{p:apriori} is an \textit{a~priori} estimate whose proof occupies the majority of this section. In \S\ref{ss:apriori.prelim} we make some preliminary definitions and give an overview of the proof of Proposition~\ref{p:apriori}.

\subsection{Preliminaries}
\label{ss:apriori.prelim}

In preparation for the proof of Proposition~\ref{p:apriori}, we introduce a richer set of colors, as follows:

\begin{dfn}[expanded alphabet of colors]\label{d:white.violet} Write $\mathscr{X}\equiv\set{\RYGB}$. Define the expanded alphabets
	\[\COLS\equiv\bigg\{\RYGB,
			\whi\equiv\SPIN{white},
			\mred\equiv\SPIN{violet}\bigg\}\,,\quad
		\mathscr{S}
			\equiv \bigg\{\RYC,\whi\bigg\}\,,\quad
		\mathscr{T}\equiv\bigg\{\RYGB,\mred\bigg\}\,.
	\]
Let $\GG=(V,F,E)$ be any (processed) $\ksat$ instance,
and let $\usi\in\set{\RYGB}^E$ denote a valid (single-copy) coloring on $\GG$. Given $\usi$, let $\uvsi\in\mathscr{S}^E$ be defined by setting
	\[\vsi_e\equiv\begin{cases}
	\whi\equiv\SPIN{white}
	&\text{if $\sigma_e\in\set{\yel,\grn,\blu$}
	and $\usi_{\delta a\setminus e}$
	contains at least two \SPIN{cyan} edges,}\\
	\yel&\textup{if $\sigma_e=\yel$,}\\
	\cya&\textup{if $\sigma_e\in\set{\grn,\blu}$.}\\
	\end{cases}
	\]
Let $\vups\in\mathscr{T}^E$ be defined by setting
	\[\ups_e
	=\begin{cases}
	\mred\equiv\SPIN{violet}
	&\textup{if $\sigma_e=\red$
	and $\usi_{\delta v\setminus e}$
	contains at least one \SPIN{red} edge,}\\
	\sigma_e\in\set{\RYGB}
	&\textup{otherwise.}
	\end{cases}
	\]
Finally, let $\vec{\col}\in\COLS^E$ be defined by setting
	\[
	\col_e \equiv
	\begin{cases}
	\whi\equiv\SPIN{white}
	&\text{if $\sigma_e\in\set{\yel,\grn,\blu$}
	and $\usi_{\delta a\setminus e}$
	contains at least two
	$\SPIN{cyan}$ edges,}\\
	\mred\equiv\SPIN{violet}
	&\text{if $\sigma_e=\red$
	and $\usi_{\delta v\setminus e}$
	contains at least one \SPIN{red} edge,}\\
	\sigma_e\in\set{\RYGB}
	&\text{otherwise.}
	\end{cases}
	\]
The new colors $\set{\whi,\mred}$ indicate edges that are ``flexible'' in some sense:
from a clause's perspective, $\SPIN{white}$ indicates an edge that can be $\SPIN{yellow}$, $\SPIN{green}$, or $\SPIN{blue}$ without violating the clause constraint. Similarly, from a variable's perspective, $\SPIN{violet}$ indicates an edge that is $\SPIN{red}$ but can also be $\SPIN{blue}$ without violating the variable constraint.
\end{dfn}

\begin{dfn}[pair empirical measure in expanded alphabet]
Throughout the remainder of this section, on a (processed) $\ksat$ instance $\GG=(V,F,E)$,
we denote pair colorings in the expanded alphabets as
	{\setlength{\jot}{0pt}\begin{align*}
	\usi &\equiv (\usi^1,\usi^2)
	\in\mathscr{X}^E\times\mathscr{X}^E\,,\\
	\uvsi
	&\equiv (\uvsi^1,
		\uvsi^2)
	\in \mathscr{S}^E\times\mathscr{S}^E\,,\\
	\vups
	&\equiv
	(\vups^1,\vups^2)
	\in \mathscr{T}^E\times\mathscr{T}^E\,,\\
	\ucol
	&\equiv (\vec{\col}^1,\vec{\col}^2)
	\in\COLS^E\times\COLS^E\,.
	\end{align*}}%
As before, we use $\omega\equiv(\omega_{\bL,j})_{\bL,j}$ to denote the pair empirical measure, which we assume to be judicious in the sense of Definition~\ref{d:judicious.pair}. We let $\bom\equiv(\bom_{\bL,j})_{\bL,j}$ denote the pair empirical measure
 of $(\usi,\ucol)$, so that each entry 
 $\bom_{\bL,j}$ is a probability measure over elements $(\sigma,\col)\in\mathscr{X}^2\times\COLS^2$. The marginal of $\bom$ on the $\mathscr{X}^2$-coordinate is given by $\omega$. Clearly, by the mappings of Definition~\ref{d:white.violet},
 each
 $\bom_{\bL,j}$ also induces a probability measure over $(\sigma,\varsigma)\in\mathscr{X}^2\times\mathscr{S}^2$
 and $(\sigma,\ups)\in\mathscr{X}^2\times\mathscr{T}^2$. We let $\zeta_{\bL,j}$ denote the marginal law of $\vsi$ under $\bom_{\bL,j}$.
\end{dfn}

Recall that Proposition~\ref{p:apriori} concerns the optimization of $\bm{\Psi}_{\DD,2}(\omega)$ over $\omega\in\bm{I}_0$. The function $\bm{\Psi}_{\DD,2}$ is defined in the discussion leading up to Lemma~\ref{l:if.neg.def}, and we review it briefly here. If $\bh=(\dbh,\hbh)$ is a (judicious) vertex empirical measure in the pair coloring model, 
then analogously to \eqref{e:config.model.first.moment.given.omega},
its contribution to the second moment is given by
	\beq\label{e:config.model.pair.version}
	\E_{\DD}\ZZ^2[\bh]
	=\underbrace{
	\Bigg\{
	\prod_{\bT} 
	\binom{n_{\bT}}{n_{\bT}\dbh_{\bT}}
	\prod_{\bL}
	\binom{m_{\bL}}{m_{\bL}\hbh_{\bL}}
	\Bigg\}
	}_{\substack{
	\text{number of colorings}\\
	\text{prior to matching}
	}}
	\underbrace{\Bigg\{
	\prod_{\bm{t}} \binom{n_{\bm{t}}}
		{ n_{\bm{t}} \pi_{\bm{t}} }
		\Bigg\}^{-1}
		}_{\substack{\text{probability of matching}\\
			\text{to respect colorings}}}
	=
	\f{\exp\{n\bm{\Phi}_{\DD,2}(\nu)\}}
		{n^{O(1)}}\,,\eeq
where $\bm{\Phi}_{\DD,2}$ is the analogue of
\eqref{e:rate.function.given.gen.degseq} for the pair model:
	\beq\label{e:rate.function.pair.model}
	\bm{\Phi}_{\DD,2}(\bh)
	=\f1n \Bigg\{
	|V| \, \E_{\dot{\DD}}[ \Ent(\dbh_{\bT}) ]
	+|F| \, \E_{\hat{\DD}}[ \Ent(\hbh_{\bL}) ]
	-|E| \, \E_{\bar{\DD}}
		[ \Ent(\pi_{\bm{t}}) ]
		\Bigg\}\,.
	\eeq
(Note that $\bh$ denotes a vertex empirical measure for the single-copy model in \eqref{e:rate.function.given.gen.degseq},
but for the pair model in \eqref{e:rate.function.pair.model}.)
Write $\nu\sim\omega$ if $\nu$ has marginals $\omega$. Then the contribution to the second moment from
any (judicious) $\omega$ is
	\[
	\E_{\DD}\ZZ^2(\omega)
	=\sum_{\nu:\nu\sim\omega}
	\E_{\DD}\ZZ^2[\nu]
	= \f{\exp\{n\bm{\Psi}_{\DD,2}(\omega)\}}
		{n^{O(1)}}\,,
	\]
where $\bm{\Psi}_{\DD,2}$ is the analogue of
\eqref{e:nu.opt} for the pair model:
	\beq\label{e:nu.opt.given.omega.pair}
	\bm{\Psi}_{\DD,2}(\omega)
	\equiv\bm{\Phi}_{\DD,2}(\optnu[\omega])\,,\quad
	\optnu[\omega]
	\equiv \argmax_\nu
	\bigg\{
	\bm{\Phi}_{\DD,2}(\nu):
	\textup{$\nu$ is consistent with $\omega$}
	\bigg\}\,.
	\eeq
We will prove Proposition~\ref{p:apriori} by analyzing
the constrained entropy maximization problems involved in the above definition of $\optnu[\omega]$. We now make two definitions which will be used throughout the section: 

\begin{dfn}\label{d:diverse} For a variable $v$ of type $\bT$,
let $\pi_v\equiv\pi_{\bT}$ denote the marginal on the (pair) frozen configuration spin $x_v$. Thus $\pi_v$ is a probability measure on $\set{\plus,\minus,\free}^2$, which can be computed from $\pi_e$ for any $e\in\delta v$: for example,
$\pi_v(\plus\plus)=\pi_e(\yel\yel)$ for any $e\in\delta v(\minus)$.
We define a \bemph{diverse variable} to be a variable $v$ of type $\bT$ such that
	\[
	\pi_v\Big(\set{\plus\minus,\minus\plus}\Big)
	\equiv
	\pi_{\bT}\Big(\set{\plus\minus,\minus\plus}\Big)
	=2\pi_{\bT}(\plus\minus)
	\ge\f14\,.
	\]
We then define a \bemph{diverse clause} to be a clause that neighbors at least $k/10$ diverse variables --- equivalently, it is a clause of type $\bL$ such that
	\[
	\sum_{j=1}^{k(\bL)}
	\mathbf{1}\bigg\{
	\pi_{\bL(j)}\Big(\set{\yel\pur,
	\pur\yel}\Big) \ge\f14
	\bigg\}
	=\sum_{j=1}^{k(\bL)}
	\mathbf{1}\bigg\{
	\pi_{\bL(j)}(\yel\pur) \ge\f18
	\bigg\}
	\ge \f{k}{10}\,.
	\]
The definition of diversity depends only on type, so we can also speak of \bemph{diverse variable types} and \bemph{diverse clause types}. Let $\mathbb{D}$ denote the collection of all diverse clause types. For a variable $v$ of type $\bT$, we write
	\beq\label{e:exp.num.nondiv.nbr.clauses}
	\notDiverse(v)
	\equiv\notDiverse(\bT)
	\equiv\sum_{\bt\in\bT}
	\sum_{\bL}
	\Ind{\bL\notin\mathbb{D}}
	\pi_{\DD}(\bL\,|\,\bt)
	\eeq
for the expected number of non-diverse clauses incident to $v$ under $\P_{\DD}$.
\end{dfn}

\begin{dfn}\label{d:heavy} Let $\EPSLIGHT>0$ be a small absolute constant. A clause type $\bL$ is termed \bemph{light} if
	\[
	\max_{1\le j\le k(\bL)}
	\omega_{\bL,j}(\red\red) \le 
	\f1{2^{k(1+\EPSLIGHT)}}\,.
	\]
Otherwise we say that $\bL$ is \bemph{heavy}.
Let $\mathbb{L}$ denote the collection of light clause types. Recall from \eqref{e:exp.num.nondiv.nbr.clauses} in Definition~\ref{d:diverse} that $\notDiverse(v)$ denotes the expected number of non-diverse clauses next to $v$. Analogously define
$\notLight(v)$ to be the expected number of non-light (i.e., heavy) clauses next to $v$.
\end{dfn}

We conclude this preliminary subsection with an overview of the proof of Proposition~\ref{p:apriori}. Recall that the goal is to show that if the neighborhood profile $\DD$ satisfies the expansion condition \eqref{e:expansion.bd} from Definition~\ref{d:expansion}, and $\omega=\omega(\DD)$ is any maximizer of $\bm{\Psi}_{\DD,2}$ over $\omega\in\bm{I}_0$ (see \eqref{e:opt.omega.of.DD}), then $\omega_{\bL,j}$ satisfies the ``\textit{a~priori}'' estimates \eqref{e:apriori} whenever $\bL(j)$ is a strongly non-defective edge type in the sense of Remark~\ref{r:defective.clauses}. In what follows, if $e=(av)$ with $\bt_e=\bt$, $j(\bt)=j$, and $\bL_a=\bL$, we will denote
$\pi_e\equiv\pi_{\bt}$,
$\omega_e\equiv\omega_{\bL,j}$,
$\bom_e\equiv\bom_{\bL,j}$,
and $\zeta_e\equiv\zeta_{\bL,j}$.

In Proposition~\ref{p:balancedeVertex} we will show that if $v$ is a non-defective variable with $\notDiverse(v)=\notLight(v)=0$, then $\omega_e$ satisfies the required estimates \eqref{e:apriori} for all $e\in\delta v$. The proof of Proposition~\ref{p:balancedeVertex}
uses a reweighting argument that builds on the result of Proposition~\ref{p:ctypes.varupdate} from Section~\ref{s:contract}. However, unlike in Section~\ref{s:contract}, the analysis of Proposition~\ref{p:balancedeVertex} uses 
some additional information concerning \SPIN{white} edges, which are only weakly dependent on the other edges sharing the same clause. In particular, in a nice clause $a$,
Lemma~\ref{l:lotsOfAs} gives that the expected number of edges $e\in\delta a$ with $\vsi_e\ne\whi\whi$ is small. If the clause is furthermore diverse and light, then
Lemma~\ref{l:lotsOfAsDiverse}
gives that the expected number of edges $e\in\delta a$
with $(\vsi_e)^i\ne\whi$ for both $i=1,2$
is very small. Both these estimates are used in the proof of Proposition~\ref{p:balancedeVertex}.

Of course, the main challenge in applying Proposition~\ref{p:balancedeVertex} 
is that it relies on the assumption
$\notDiverse(v)=\notLight(v)=0$ --- that is, \emph{all} clauses neighboring any variable of this type must be diverse and light. Thus a large part of this section is devoted to proving estimates leading to Proposition~\ref{p:noNonDiverse}, which uses an expansion argument (under the condition~\eqref{e:expansion.bd}) to show that \emph{all} strongly non-defective clauses must be diverse and light. Proposition~\ref{p:balancedeVertex} can then be applied to give the conclusion of Proposition~\ref{p:apriori}.

The proof of Proposition~\ref{p:noNonDiverse} proceeds roughly as follows. For the purposes of this overview we will ignore the presence of defective variables. For any $\omega\in\bm{I}_0$, the empirical distribution of frozen spins $x_v\in\set{\minus,\plus,\free}^2$ among all $v\in V$ must give weight roughly $1/4$ to each $x\in\set{\minus,\plus}^2$. Of course, these may not be evenly distributed among the different variable types, so this does not imply that all variable types are diverse. However, an easy application \eqref{e:diverse.Markov} of Markov's inequality shows that at least a quarter of all variables must be diverse. The expansion condition
\eqref{e:expansion.bd} then implies that most clauses are diverse, since most clauses will be incident to more than $k/10$ diverse variables.
We then show in Lemma~\ref{l:large11Case} and Corollary~\ref{c:RRFreq} that if a variable $v$
has a small value of $\notDiverse(v)$, then it will also have a small value of $\notLight(v)$.
We show in Proposition~\ref{p:DIVERSE.VAR} that if $\notDiverse(v)$ and $\notLight(v)$ are both small, then $v$ will be diverse. To summarize, write $S$ for the set of non-diverse variables, $\notDiverse$ for the set of non-diverse clauses, and $\notLight$ for the set of heavy clauses. Then
\begin{enumerate}[--]
\item $|S|$ can be bounded in terms of $|\notDiverse|+|\notLight|$ by Proposition~\ref{p:DIVERSE.VAR};
\item $|\notLight|$ can be bounded in terms
of $|\notDiverse|$ by Lemma~\ref{l:large11Case} and Corollary~\ref{c:RRFreq}; and
\item $|\notDiverse|$ can be bounded in terms of $|S|$ by the expansion condition \eqref{e:expansion.bd}.
\end{enumerate}
Combining these gives a bound for $|S|$ in terms of $|S|$ itself, which we find is satisfied only if $|S|=0$. This yields the conclusion of Proposition~\ref{p:noNonDiverse}. 

This concludes our overview for the proof of Proposition~\ref{p:apriori}, and we now turn to the details of the proof. For the reader's reference, the dependency diagram of results in this section is given in Figure~\ref{f:apriori}.

\begin{figure}[h!]
\begin{tikzpicture}[node distance = 3cm,auto]
\footnotesize 
\node[block,align=center] (lotsOfAs)
	at (6,14) {Lemma~\ref{l:lotsOfAs}:
	if $a$ is a nice clause
	then $\zeta_e(\whi\whi)$ is large for all $e\in\delta v$};
\node[block,align=center] (indifferent) at (-6,14)
	{Lemma~\ref{l:indifferent}:
	resampling argument to deduce bounds on $\omega_e$ from bounds on $\pi_e$};
\node[block,align=center] (lotsOfAsDiverse) at (-6,10)
	{Lemma~\ref{l:lotsOfAsDiverse}: 
	if a clause $a$ is diverse and light,
	then $\zeta_e(\set{\RYC}^2)$ is small for all 
	$e\in\delta a$};
\node[block] (VarUpdate) at (-6,5.5)
	{Proposition~\ref{p:ctypes.varupdate}:
	analysis of variable update used in proof of 
	Lemma~\ref{l:rrUnforced}};
\node[block] (ZETAbalancedVx) at (-1.5,3.5)
	{Lemma~\ref{l:ZETA.conditions.balanced.vertex}:
	if $v$ is a non-defective variable
	with $\zeta_e(\whi\whi)$ large and
	$\zeta_e(\set{\RYC}^2$ small for all $e\in\delta v$,
	then $\omega_e$ near product for all $e\in\delta v$};
\node[block] (ZETAsingleEdge) at (-1.5,5.5)
	{Lemma~\ref{l:single.edge.sigma.varsigma}: analysis of an edge update used in proof of
	Lemma~\ref{l:rrUnforced}};
\node[block,align=center] (rrUnforced) at (1.5,12)
	{Lemma~\ref{l:rrUnforced}:
conditions to guarantee that
most occurrences of $\sigma^i=\red$ we have $\ups^i=\mred$};
\node[block,align=center] (ryUnforced) at (-3,12) {Lemma~\ref{l:ryUnforced}: conditions to guarantee that
most occurrences of $\sigma^i=\red$ we have $\ups^i=\mred$};
\node[block,align=center] (ForcedRR) at (6,12) {Lemma~\ref{l:ForcedRR}: if a clause $a$ is diverse, then $\ups_e=\mred\mred$ occurs rarely
for all $e\in\delta a$};
\node[block,align=center] (largeRRcase) at (0,10)
	{Lemma~\ref{l:large11Case}: if $v$ is non-defective with $\notDiverse(v)$ small and $\pi_v(\plus\plus)$ bounded below, then
$\omega_e(\red\red)$ is small for all $e\in\delta v$};
\node[block,align=center] (RRFreq) at (6,10)
	{Corollary~\ref{c:RRFreq}: if $v$ is non-defective with $\notDiverse(v)$ small, then
	$\omega_e(\red\red)$ for $e\in\delta v$
	is small on average};
\node[block,align=center] (expansion) at (-6,-1)
	{Lemma~\ref{l:expand.on.types}: the processed $\ksat$ graph satisfies the
	expansion
	condition with high probability};
\node[block,align=center] (balancedeVertex) at (1,1) {Proposition~\ref{p:balancedeVertex}: 
if $v$ is a non-defective variable with $\notDiverse(v)=\notLight(v)=0$,
then $\omega_e$ is near product
for all $e\in\delta v$};
\node[block,align=center] (noNonDiverse) at (6,1)
	{Proposition~\ref{p:noNonDiverse}: 
	under expansion condition,
	all non-defective variables are diverse,
	all strongly non-defective clauses are light};
\node[block,align=center] (APRIORI) at (1,-1) {Proposition~\ref{p:apriori}: if $v$ is strongly non-defective, then $\omega_e$ is near product for all $e\in\delta v$};
\node[block,align=center] (entropyComp) at (1,7.5)
	{Proposition~\ref{p:DIVERSE.VAR}:
	if $v$ is a nondefective variable with
	$\notDiverse(v),\notLight(v)$ small,
	then $\pi_v$ is roughly uniform
	on $\set{\minus,\plus}^2$};
\path [line] (balancedeVertex) -- (APRIORI);
\path [line] (noNonDiverse) -- (APRIORI);
\path [line] (entropyComp) -- (noNonDiverse);
\path [line] (RRFreq) -- (noNonDiverse);
\path [line] (largeRRcase) -- (5,7.5) -- (noNonDiverse);
\path [line] (largeRRcase) -- (RRFreq);
\path [line] (rrUnforced) -- (RRFreq);
\path [line] (ForcedRR) -- (RRFreq);
\path [line] (ZETAbalancedVx) -- (balancedeVertex);
\path [line] (ZETAsingleEdge) -- (ZETAbalancedVx);
\path [line] (VarUpdate) -- (ZETAbalancedVx);
\path [line] (lotsOfAsDiverse) -- (entropyComp);
\path [line] (indifferent) -- (lotsOfAsDiverse);
\path [line] (indifferent)
	-- (-5.5,11.25) -- (largeRRcase);
\path [line] (ryUnforced) -- (largeRRcase);
\path [line] (lotsOfAsDiverse) -- (largeRRcase);
\path [line] (rrUnforced) -- (largeRRcase);
\path [line] (ForcedRR) -- (largeRRcase);
\path [line] (lotsOfAsDiverse)
	-- (-4,6.5) -- (1,6.5) -- (balancedeVertex);
\path [line] (indifferent)
	-- (2,14) -- (ForcedRR);
\end{tikzpicture}
\caption{Dependency diagram of results of Section~\ref{s:burnin}.
Lemma~\ref{l:lotsOfAs} is a basic estimate used in the proofs of several subsequent claims, so the arrows leaving Lemma~\ref{l:lotsOfAs} are not shown to avoid cluttering the diagram further. Lemma~\ref{l:expand.on.types} and Proposition~\ref{p:apriori}
were both used in the proof of Proposition~\ref{p:second.moment.judicious}.}
\label{f:apriori}
\end{figure}

\subsection{Entropy maximization around non-forcing clauses} \label{ss:white.diverse}

The goal of this subsection is to prove that, for clause types $\bL$ satisfying certain conditions, $\omega_{\bL,j}(\varsigma\ne\whi\whi)$ must be small, and $\omega_{\bL,j}(\varsigma^1\ne\whi, \varsigma^2\ne\whi)$ must be even smaller, for all $1\le j\le k(\bL)$. In later subsections we will prove estimates restricted to \SPIN{white} edges --- these results will transfer easily to estimates concerning all edges, since the results of this subsection show that most edges are \SPIN{white}. The current subsection is organized as follows:
\begin{enumerate}[--]
\item In Lemma~\ref{l:lotsOfAs} we show that if $\bL$ is a nice clause, then $\omega_{\bL,j}(\varsigma\ne\whi\whi)$ must be small (at most $O(k^2/2^k)$) for all $1\le j\le k(\bL)$.
\item In Lemma~\ref{l:indifferent} we show that if $\bt$ is a non-compound edge type, and $\sigma\in\set{\yel,\blu}^2$, then a lower bound on $\pi_{\bt}(\sigma)$ implies a comparable lower bound on $\omega_{\bL,j}(\sigma)$ for all $\bL(j)=\bt$.

\item In Lemma~\ref{l:lotsOfAsDiverse} we prove that
if a clause type $\bL$ is nice, diverse, and light, then
$\omega_{\bL,j}(\varsigma^1\ne\whi, \varsigma^2\ne\whi)\le 2^{-k(1+\EPSLIGHT)}$ for all $1\le j\le k(\bL)$. The proof of this result makes use of 
Lemma~\ref{l:indifferent}, which gives
a lower bound on $\omega_{\bL,j}(\set{\yel\blu,\blu\yel})$ for at least $k/10$ indices $1\le j\le k(\bL)$.
\end{enumerate}
We now turn to the precise statements and proofs.

\begin{lem}\label{l:lotsOfAs} 
Let $a$ be a nice clause of type $\bL$ (meaning all its incident edges are nice in the sense of Definition~\ref{d:nice}). Given $\omega\in\bm{I}_0$, let $\hbh_{\bL}$ be the probability measure on valid (pair) colorings $\usi_{\delta a}$ which maximizes entropy subject to edge marginals $(\omega_{\bL,j})_j$. By Definition~\ref{d:white.violet},
each $\usi_{\delta a}$ maps to a configuration
$\uvsi_{\delta a}$, so $\hbh_{\bL}$ also induces a probability measure on $(\mathscr{S}^2)^{\delta a}$. Under this measure we have
	\[
	\bom_{\bL,j}\Big(\varsigma=\whi\whi\Big)
	=\hbh_{\bL}\Big(\varsigma_j=\whi\whi\Big)
	\ge 1-O\bigg(\f{k^2}{2^k}\bigg)
	\]
for all $1\le j\le k(\bL)$.

\begin{proof} For the proof we will mostly suppress $\bL$ from the notation, and write $\hbh\equiv\hbh_{\bL}$. We claim that for this proof it suffices to take $\sigma$ in the reduced alphabet $\set{\RYC}$. Indeed, each $\omega_{\bL,j}$ is a measure over $\set{\RYGB}^2$, and naturally induces a measure $\set{\RYC}^2$, which we temporarily denote $\tilde{\omega}_{\bL,j}$. Let
$\tilde{\nu}\equiv\tilde{\nu}_{\bL}$ be the probability measure on valid (pair) colorings $\tilde{\usi}_{\delta a}\in(\set{\RYC}^{\delta a})^2$ which maximizes entropy subject to edge marginals $(\tilde{\omega}_{\bL,j})_j$.
Since the clause factor \eqref{e:indicator.of.valid.clause.coloring} does not distinguish between \SPIN{green} and \SPIN{blue}, the optimal $\hbh\equiv\hbh_{\bL}$ must take the form
	\beq\label{e:bg.given.cyan}
	\hbh_{\bL}(\usi_{\delta a})
	= \tilde{\nu}(\tilde{\usi}_{\delta a})
	\prod_{j=1}^{k(\bL)}
	\f{\omega_{\bL,j}(\sigma_j)}
	{\tilde{\omega}_{\bL,j}(\tilde{\sigma}_j)}\,,
	\eeq
where $\usi_{\delta a}\in(\set{\RYGB}^{\delta a})^2$ and $\tilde{\usi}_{\delta a}$ is its representative in $(\set{\RYC}^{\delta a})^2$. Thus it suffices to study $\tilde{\nu}$. For the remainder of the proof we will abuse notation slightly and write $\hbh,\usi_{\delta a}$ when we technically mean $\tilde{\nu},\tilde{\usi}_{\delta a}$. The remainder of the proof is divided into a few numbered steps. \smallskip 

\noindent\bemph{Step 1. Reduction to consideration of non-forcing clause colorings.}
The single-copy marginals of $\omega$ are given by the canonical measure $\starpi$. Since $\bL$ is assumed to be nice, it follows from Definition~\ref{d:nice} that
	\beq\label{e:lotsOfAs.hyp}
	\starpi_{\bL(j)}(\red)
	\le\f{O(1)}{2^k}\,,\quad
	\bigg|\starpi_{\bL(j)}(\cya)
	-\f12\bigg|
	\le O\bigg(
	\f{1}{2^{k/10}}\bigg)\,.
	\eeq
Let $\SPIN{U}$ denote the set of all \bemph{valid} pair colorings $\usi_{\delta a} \in\set{\yel,\cya}^{2K}$, where we abbreviate $K\equiv k(\bL)$. Note that if $\usi_{\delta a}$ is a valid coloring belonging to the complement $\SPIN{U}^c$, then $\usi_{\delta a}$ must have at least one edge colored $\SPIN{red}$ in at least one of the two copies. It follows by the judicious condition together with \eqref{e:lotsOfAs.hyp} that
	\beq\label{e:any.red.is.unlikely}
	\hbh(\SPIN{U}^c)
	\le
	\sum_{i=1,2}
	\sum_{j=1}^K
	(\omega_{\bL,j})^i(\red)
	= 2 \sum_{j=1}^K
	\starpi_{\bL(j)}(\red)
	\stackrel{\eqref{e:lotsOfAs.hyp}}{\le}
	O\bigg( \f{k}{2^k}\bigg)\,.
	\eeq
Define the conditional measure $\mu(\usi_{\delta a})\equiv\hbh(\usi_{\delta a}\,|\,\SPIN{U})$. Let $\mu_j$ be the marginal of $\mu$ on the $j$-th edge in $\delta a$, so $\mu_j$ is a probability measure on $\set{\yel,\cya}^2$. Since $\hbh(\SPIN{U})$ is close to one by \eqref{e:any.red.is.unlikely},
the single-copy marginals of $\mu_j$ must be close to $\starpi_{\bL(j)}$; in particular, we must have
	\beq\label{e:cond.yb.single.copy.mgls}
	(\mu_j)^i(\cya)
	=\f{ \starpi_{\bL(j)}(\cya)
		-\hbh(\SPIN{U}^c) \hbh( (\sigma_j)^i=\cya\,|\,\SPIN{U}^c )}{1 - \hbh(\SPIN{U}^c)}
	\stackrel{\eqref{e:any.red.is.unlikely}}{=} 
	\starpi_{\bL(j)}(\cya) + O\bigg(\f1{2^k}\bigg)
	\stackrel{\eqref{e:lotsOfAs.hyp}}{=}
	\f12 +O\bigg(
	\f{1}{2^{k/10}}\bigg)\,.
	\eeq
Since $\hbh=\hbh_{\bL}$ maximizes entropy subject to marginals $(\omega_{\bL,j})_j$, it must be that $\mu$ maximizes entropy subject to marginals $(\mu_j)_j$. Consequently, by the method of Lagrange multipliers, there must exist probability measures $q_j$ over $\set{\yel,\cya}^2$ (for $1\le j\le K$) such that
	\beq\label{e:cond.yb.q}
	\mu(\usi_{\delta a})
	=
	\f{\Ind{\usi_{\delta a}\in\SPIN{U}}
	Q(\usi_{\delta a})}{Q(\SPIN{U})},
	\quad
	Q(\usi_{\delta a})
	\equiv\prod_{j=1}^K q_j(\sigma_e)\,.
	\eeq
In the next step we will estimate the $q_j$ to show that under $\mu$, each edge is $\whi\whi$ with probability close to one. Since $\mu$ takes up most of the mass of $\hbh$ by \eqref{e:any.red.is.unlikely}, the result will follow.\smallskip

\noindent\bemph{Step 2. Construction of Lagrangian weights.} We will iteratively construct a sequence $q_{j,t}$ that converges to the desired $q_j$
of \eqref{e:cond.yb.q} in the limit $t\to\infty$. We initialize $q_{j,0}=\mu_j$ for all $1\le j\le K$. Analogously to \eqref{e:cond.yb.q} let
	\[\mu_t(\usi_{\delta a})
	=
	\f{\Ind{\usi_{\delta a}\in\SPIN{U}}
	Q_t(\usi_{\delta a})}{Q_t(\SPIN{U})},
	\quad
	Q_t(\usi_{\delta a})
	\equiv\prod_{j=1}^K q_{j,t}(\sigma_e)\,.
	\]
Let $\mu_{j,t}$ denote the marginal of $\mu_t$ on the $j$-th edge. Writing 
$\usi_{-j}\equiv (\sigma_\ell)_{\ell\in[K]\setminus j}$, we have
	\beq\label{e:q.times.prob.of.remainder}
	\mu_{j,t}(\sigma)
	= \f1{\bar{z}_{j,t}}
	q_{j,t}(\sigma)
	\overbrace{
	\sum_{\usi_{-j}}
	\Ind{\usi\in\SPIN{U}}
	\prod_{\ell\in[K]\setminus j}
	q_{\ell,t}(\sigma_\ell)
	}^{\textup{denote this }\xi_{j,t}(\sigma)}\,.
	\eeq
Note that $\xi_{j,t}(\sigma)$ has a simple expression for each $\sigma\in\set{\yel,\cya}^2$: for instance, for $\sigma=\cya\cya$ we have
	\[
	\xi_{j,t}(\cya\cya)
	=1
	-\sum_{i=1,2}
	\prod_{\ell\in[K]\setminus j}
	(q_{\ell,t})^i(\yel)
	+\prod_{\ell\in[K]\setminus j}
	q_{\ell,t}(\yel\yel)\,,
	\]
and we have similar expressions for $\sigma\in\set{\yel\yel,\yel\cya,\cya\yel}$. It is easily verified that
$1-k^{O(1)}/2^k\le \xi_{j,0}(\sigma)\le1$
for all $\sigma\in\set{\yel,\cya}^2$, and substituting this estimate into \eqref{e:q.times.prob.of.remainder} gives (crudely)
	\beq\label{c:lotsOfAs.base.case}
	\bigg\| \f{\mu_{j,0}}{\mu_j}-1
		\bigg\|_\infty
	\stackrel{\eqref{e:q.times.prob.of.remainder}}{=}
	\bigg\| \f{\xi_{j,0}}{ \bar{z}_{j,0}}-1
	\bigg\|_\infty \le \f1{2^{2k/3}}\,.
	\eeq
Moreover, it follows by a straightforward calculation that 
	\beq\label{e:d.xi.d.q.bound}
	\bigg|\f{\pd \xi_{j,t}(\sigma)}
	{ \pd q_{\ell,t}(\sigma') }\bigg|
	\le \f{k^{O(1)}}{2^k}
	\eeq
for all $j\ne \ell$ and all $\sigma,\sigma'\in\set{\yel,\cya}^2$. We now define
$q_{j,t+1}$ to be the probability measure such that
	\beq\label{e:q.mu.update.rule}
	q_{j,t+1}(\sigma)
	\cong \f{\mu_j(\sigma)}{\xi_{j,t}(\sigma)}
	\stackrel{\eqref{e:q.times.prob.of.remainder}}
		{\cong}
	\f{\mu_j(\sigma)}{\mu_{j,t}(\sigma)}
		q_{j,t}(\sigma)\,.
	\eeq
Let $z_{j,t+1}$ denote the normalizing constant such that
	\beq\label{e:z.j.t.plus.one}
	q_{j,t+1}(\sigma)
	=\f{\mu_j(\sigma)}{\mu_{j,t}(\sigma)}
		\f{q_{j,t}(\sigma)}{z_{j,t+1}}\,.
	\eeq
Now suppose inductively that for all $t\ge0$ we have
	\beq\label{c:lotsOfAs.induction}
	\bigg\|
	\f{\mu_{j,t} }{ \mu_j}
	-1\bigg\|_\infty
	\le \f1{(2^{2k/3})^{t+1}}\,,\eeq
where the base case $t=0$ is given by \eqref{c:lotsOfAs.base.case}. Note that substituting \eqref{c:lotsOfAs.induction} into \eqref{e:z.j.t.plus.one} implies that $q_{j,t}$ is close to $z_{j,t+1}q_{j,t+1}$; since all the $q$'s are probability measures, it further implies that $z_{j,t+1}$ is close to one. Thus, for $t\ge0$,
the inductive hypothesis \eqref{c:lotsOfAs.induction} leads to
	\beq\label{e:induction.implies.xi.err.bound}
	\Big\|\xi_{j,t+1}-\xi_{j,t}\Big\|_\infty
	\stackrel{\eqref{e:d.xi.d.q.bound}}{\le}
	\f{k^{O(1)}}{2^k}
	\Big\|q_{j,t+1}-q_{j,t}\Big\|_\infty
	\stackrel{\eqref{e:z.j.t.plus.one}}{=}
	\f{k^{O(1)}}{2^k}
	\Bigg\|q_{j,t}
	\bigg(
	\f{\mu_j}{\mu_{j,t}}
		\f{1}{z_{j,t+1}}
	-1\bigg)\Bigg\|_\infty
	\stackrel{\eqref{c:lotsOfAs.induction}}{\le}
	\f{k^{O(1)}}{2^k}
	\f1{(2^{2k/3})^{t+1}}\,.
	\eeq
Recall that $\mu_{j,t+1} = q_{j,t+1} \xi_{j,t+1} / \bar{z}_{j,t+1}$ by \eqref{e:q.times.prob.of.remainder}, while
\eqref{e:q.mu.update.rule} implies that
there exists a normalizing constant $\acute{z}_{j,t+1}$ such that $\mu_j = q_{j,t+1} \xi_{j,t}/\acute{z}_{j,t+1}$. It follows that
	\[
	\bigg\| \f{\mu_{j,t+1}}{\mu_j}-1\bigg\|_\infty
	=
	\bigg\| 
	\f{\xi_{j,t+1} / \bar{z}_{j,t+1}}
		{\xi_{j,t} / \acute{z}_{j,t+1}}-1
	\bigg\|_\infty
	\stackrel{\eqref{e:induction.implies.xi.err.bound}}
		{\le}
	\f1{(2^{2k/3})^{t+2}}\,,
	\]
which verifies the inductive hypothesis \eqref{c:lotsOfAs.induction}.\smallskip

\noindent
\bemph{Step 3. Conclusion.}
Let $q_j=\lim_{t\to\infty} q_{j,t}$, where it is clear from \eqref{c:lotsOfAs.induction} that the limit is well-defined. Moreover, summing \eqref{c:lotsOfAs.induction} over $t\ge0$ implies
	\begin{align*}
	\bigg\|\f{q_j}{\mu_j}-1\bigg\|_\infty
	&=\bigg\|\f{q_j}{q_{j,0}}-1\bigg\|_\infty
	\le
	\sum_{t\ge0}
	\bigg\|\f{q_{j,t+1}}{q_{j,t}}-1\bigg\|_\infty\\
	&\stackrel{\eqref{e:z.j.t.plus.one}}
		{\le} O\Bigg(
	\sum_{t\ge0}
	\bigg\|\f{\mu_{j,t}}{\mu_j}
		-1\bigg\|_\infty
		\Bigg)
	\stackrel{\eqref{c:lotsOfAs.induction}}{\le}
	O\Bigg(
	\sum_{t\ge0}
	\f1{(2^{2k/3})^{t+1}}\Bigg)
	\le \f1{2^{k/2}}\,.
	\end{align*}
We can then straightforwardly derive from
\eqref{e:cond.yb.single.copy.mgls}, \eqref{e:cond.yb.q}, and
the last bound that 	
	\[\mu\Big(\varsigma_j=\whi\whi\Big)
	\ge 1- O\bigg(\f{k^2}{2^k}\bigg)\]
for all $1\le j\le K$. The result follows since $\mu$ is the measure $\hbh$ conditioned on event $\SPIN{U}$,
and we know that
 $\hbh(\SPIN{U}^c)$ is negligible by \eqref{e:any.red.is.unlikely}.
\end{proof}
\end{lem}

\begin{lem}\label{l:indifferent} Suppose $\omega\in\bm{I}_0$, and let $\pi\equiv(\pi_{\bt})_{\bt}$ be the marginal of $\omega$ --- each $\pi_{\bt}$ is obtained by averaging over $\omega_{\bL,j}$ such that $\bL\ni\bt$ (see Definition~\ref{d:empirical.measures.of.colors}). Let
	\[
	\optnu[\pi]
	\equiv
	\argmax_\nu
	\bigg\{
	\bm{\Phi}_{\DD,2}(\nu):
	\textup{$\nu$ is consistent with $\pi$}
	\bigg\}\,.
	\]
This is a relaxation of
\eqref{e:nu.opt.given.omega.pair}, since if $\nu$ is consistent with $\omega$ then it is also consistent with $\pi$. The relaxation is also a strictly convex problem, so the maximizer $\optnu[\pi]$ is uniquely defined. Let $\acute{\omega}$ denote the marginals of $\optnu[\pi]$. Then, for every $\sigma\in\set{\yel,\blu}^2$ and every non-compound edge $\bt$ with $\pi_{\bt}(\sigma)\ge 2^{-k/15}$, we have
	\[\min\bigg\{
	\f{\acute{\omega}_{\bL,j}(\sigma)}{\pi_{\bt}(\sigma)}
	:\bL(j)=\bt\bigg\}
	\ge \f18\,,
	\]
where $j\equiv j(\bt)$. (If $\bt$ is a compound edge type then there is only one clause type $\bL$ with $\bL(j)=\bt$, so in this case $\omega_{\bL,j}=\pi_{\bt}$ and there is nothing to prove.)

\begin{proof} The proof follows a familiar outline: we first reduce to a simplified constrained entropy maximization problem, then estimate the Lagrangian weights solving that problem to derive the conclusion.\smallskip

\noindent
\bemph{Step 1. Simplified entropy maximization problem.}
From the expression \eqref{e:rate.function.pair.model}
for $\bm{\Phi}_{\DD,2}$, we see that to optimize
$\bm{\Phi}_{\DD,2}(\nu)$ given fixed $\pi$, we can optimize separately over the variable and clause empirical measures, $\dbh$ and $\hbh$. We can therefore consider $\hbh$ alone, since it determines $\acute{\omega}$. The optimal $\hbh$ is given by
	\beq\label{e:max.entropy.hbh.given.pi}
	\hbh=\argmax_{\hbh}
	\Bigg\{
	\E_{\hat{\DD}}[\Ent(\hbh_{\bL})]
	\equiv
	\sum_{\bL}\hat{\DD}(\bL) \Ent(\hbh_{\bL})
	: \textup{$\hbh$ is judicious and 
	consistent with $\pi$}
	\Bigg\}
	\eeq
--- this is because $\E_{\hat{\DD}}[\Ent(\hbh_{\bL})]$ is the only term of \eqref{e:rate.function.pair.model} that varies with $\hbh$ when $\pi$ is fixed. Equivalently, recalling the combinatorial calculation \eqref{e:config.model.pair.version}, 
$\hbh$ must satisfy
	\beq\label{e:max.entropy.hbh.given.pi.comb}
	\hbh=\argmax_{\hbh}
	\Bigg\{
	\prod_{\bL}
	\binom{m_{\bL}}{m_{\bL}\hbh_{\bL}}
	\Bigg\}\,,
	\eeq
that is to say, $\hbh$ is the empirical measure of clause colorings $\usi_{\delta a}$ that maximizes entropy subject to marginals $\pi$.

Now, as in the statement of the lemma, let us fix a non-compound edge type $\bt$ such that $\pi_{\bt}(\sigma)\ge 2^{-k/20}$. Denote $j=j(\bt)$, and recall that we write $\bL\ni\bt$ if and only if $\bL(j)=\bt$. 
Given $\GG=(V,F,E)$, define the subset of clauses
	\[
	F(\bt)
	= \bigg\{ a\in F : \bL_a \ni\bt\bigg\}
	\subseteq F\,.
	\]
Note for all $\bL\ni\bt$
the clause width $k(\bL)$ equals the same value $K\in\set{k-1,k}$, since $\bL$ must be compatible with $\bt$. Recall from Definition~\ref{d:white.violet} that a clause coloring
$\usi\in(\mathscr{X}^K)^2$
defines an element
$\uvsi\in(\mathscr{S}^K)^2$. Let
	\[
	\SPIN{N}
	\equiv
	\SPIN{N}_j
	\equiv\bigg\{
	\usi \in (\mathscr{X}^K)^2
	: \sigma_j \in\set{\yel,\blu}^2
	\textup{ and }
	\varsigma_j = \whi\whi
	\bigg\}\,.
	\]
Then, for any (pair) coloring $\usi$ on $\GG$, let
	\[
	F(\bt,\usi)
	\equiv
	\bigg\{ a\in F(\bt) : 
		\usi_{\delta a}\in\SPIN{N}\bigg\}
	\subseteq F(\bt)\,.
	\]
\bemph{Within $F(\bt,\usi)$ only,} because each clause has $\varsigma_j=\whi\whi$, we are free to reassign the value of $\sigma_j$ to any other color in $\set{\yel,\blu}^2$, provided we continue to respect $\pi$ and the judicious constraints. Let $\bP\equiv\bP_{\bt}$ be the probability measure
	\[
	\bP\bigg(\bL,\usi_{\delta a}=\usi\bigg)
	=\pi_{\DD}(\bL\,|\,\bt)
	\cdot \hbh_{\bL}\bigg(\usi_{\delta a}=\usi
		\bigg)\,,
	\]
so $\bP$ represents the empirical measure of clause types and colorings within $F(\bt)$. 
Note that marginal of $\bP$ on $\bL$ is $\bP(\bL)=\pi_{\DD}(\bL\,|\,\bt)$, while the marginal on $\sigma_j$ (for $j=j(\bt)$) is $\bL(\sigma_j=\sigma)=\pi_{\bt}(\sigma)$. Let $\zeta$ be the empirical measure for the $j$-th edges of the clauses in $F(\bt,\usi)$, i.e.,
	\[
	\zeta(\sigma,\bL)
	= \bP\bigg(
		\bL,\sigma_j=\sigma
		\,\bigg|\,\SPIN{N}
		\bigg)\,,
	\]
where $\SPIN{N}$ is shorthand for the event that $\usi\in\SPIN{N}$. If $\hbh$ is the maximizer as in \eqref{e:max.entropy.hbh.given.pi} or \eqref{e:max.entropy.hbh.given.pi.comb}, then $\zeta$ must satisfy
	\beq\label{e:final.reduction.for.going.from.t.to.L}
	\zeta=\argmax_\zeta\left\{
	\begin{array}{c}\Ent(\zeta):
	\zeta(\bL) = \bP(\bL\,|\,\SPIN{N}) \ \forall \bL,\\
	\zeta(\sigma) = \bP(\sigma_j=\sigma\,|\,\SPIN{N})
		\ \forall \sigma\in\set{\yel,\blu}^2,\\
	\zeta(\tau\,|\,\bL)
		=\bP((\sigma_j)^i=\tau\,|\,\SPIN{N},\bL)
		\ \forall \tau\in\set{\yel,\blu}
	\end{array}\right\}\,.
	\eeq
By the method of Lagrange multipliers, there exist (real-valued) weights
$\gm$, $\beta(\bL)$ , $(\beta_{\bL})^1$, $(\beta_{\bL})^2$ 
such that
	\beq\label{e:zeta.lagrangian}
	\zeta(\bL,\sigma)\cong\exp\Bigg\{
		\gm\Ind{\sigma^1\ne\sigma^2}
			+\beta(\bL)
		+2\sum_{i=1,2}
			(\beta_{\bL})^i
				\Ind{\sigma^i=\blu}
		\Bigg\}\,,
	\eeq
where $\gm$ can be chosen independently of $\bL$ since its purpose is to enforce the constraint on $\zeta(\sigma)$. (The weights $(\beta_{\bL})^i$ are multiplied by a factor of two for convenience in subsequent calculations.)\smallskip

\noindent\bemph{Step 2. Estimation of marginals for \eqref{e:zeta.lagrangian}.} Note that since $\bt$ was assumed to be a non-compound edge type, any clause type $\bL\ni\bt$ must be nice. It then follows from the judicious condition, Definition~\ref{d:nice}, and Lemma~\ref{l:lotsOfAs} that for all $\bL\ni\bt$ we have
	\beq\label{e:mostly.N}
	1\ge \hbh_{\bL}(\SPIN{N})
	\ge
	\hbh_{\bL}\Big(\varsigma_j=\whi\whi\Big)
	-\sum_{i=1,2}
	\overbrace{\hbh_{\bL}\Big( (\sigma_j)^i=\grn\Big)}^{\starpi_{\bt}(\grn)}
	\ge
	1-O\bigg(\f{k^2}{2^k}\bigg)\,.
	\eeq
As a result, the quantities appearing in \eqref{e:final.reduction.for.going.from.t.to.L} can be written more explicitly and estimated as follows:
	\[\zeta(\bL)=\bP(\bL\,|\,\SPIN{N})
	=\f{ \pi_{\DD}(\bL\,|\,\bt)
		\hbh_{\bL}(\SPIN{N})}
		{\sum_{\bL'} \pi_{\DD}(\bL'\,|\,\bt)
		\hbh_{\bL'}(\SPIN{N})}
	\stackrel{\eqref{e:mostly.N}}{=}
	\pi_{\DD}(\bL\,|\,\bt)
	\Bigg\{ 1- O\bigg( \f{k^2}{2^k}\bigg)\Bigg\}\,,
	\]
i.e., the clause type proportions within $F(\bt,\usi)$ are close to those within $F(\bt)$. Next, for each $\sigma\in\set{\yel,\blu}^2$, we have
	\[\zeta(\sigma)=
	\bP\Big(\sigma_j=\sigma\,\Big|\,\SPIN{N}\Big)
	= \f{\pi_{\bt}(\sigma)
		-\bP(\sigma_j=\sigma,\SPIN{N}^c)}
		{\bP(\SPIN{N})}
	\stackrel{\eqref{e:mostly.N}}{=}
	\pi_{\bt}(\sigma)
	-O\bigg( \f{k^2}{2^k}\bigg)\,,
	\]
where we note that the right-hand side
must be positive by the assumption that $\pi_{\bt}\ge 2^{-k/20}$. 
Finally, for each $\sigma\in\set{\yel,\blu}^2$ and each $\bL$, we have
	\beq\label{e:zeta.close.to.omega}
	\zeta(\sigma\,|\,\bL)=
	\bP\Big(\sigma_j=\sigma
		\,\Big|\,\SPIN{N},\bL\Big)
	= \f{\omega_{\bL,j}(\sigma) 
		-\bP(\sigma_j=\sigma,\SPIN{N}^c\,|\,\bL)
		}{\bP(\SPIN{N}\,|\,\bL)}
	\stackrel{\eqref{e:mostly.N}}{=} 
	\omega_{\bL,j}(\sigma) 
	-O\bigg( \f{k^2}{2^k}\bigg)\,,\eeq
where we have not yet shown the right-hand side to be positive. Taking the marginal on the $i$-th copy gives
	\[
	\zeta( \sigma^i=\tau\,|\,\bL)=
	\bP\Big((\sigma_j)^i=\tau
		\,\Big|\,\SPIN{N},\bL\Big)
	= \f{\starpi_{\bt}(\tau)
		-\bP((\sigma_j)^i=\tau,\SPIN{N}^c\,|\,\bL)
		}{\bP(\SPIN{N}\,|\,\bL)}
	\stackrel{\eqref{e:mostly.N}}{=} 
	\starpi_{\bt}(\tau)
	-O\bigg( \f{k^2}{2^k}\bigg)
	\]
for each $\tau\in\set{\yel,\blu}$.
This concludes our estimates for the quantities appearing in \eqref{e:final.reduction.for.going.from.t.to.L}. Next, we note that by the judicious condition, for all $\bL,j$ we have
	\[\omega_{\bL,j}(\yel\blu)-\omega_{\bL,j}(\blu\yel)
	=\bigg\{
	\omega_{\bL,j}(\yel\blu)+\omega_{\bL,j}(\yel\yel)
	\bigg\}
	-\bigg\{
	\omega_{\bL,j}(\blu\yel)+\omega_{\bL,j}(\yel\yel)
	\bigg\}
	=\starpi_{\bt}(\yel)-\starpi(\yel)
	=0\,.
	\]
By the judicious condition together with the assumption that $\bt=\bL(j)$ is nice, we also have
	\[
	\omega_{\bL,j}(\yel\yel)-\omega_{\bL,j}(\blu\blu)
	=\bigg\{\omega_{\bL,j}(\yel\yel)
	+\omega_{\bL,j}(\blu\yel)\bigg\}
	-\bigg\{\omega_{\bL,j}(\blu\blu)
		+\omega_{\bL,j}(\blu\yel)
		\bigg\}
	=\starpi(\yel)-\starpi(\blu)
	=O\bigg( \f1{2^{k/10}}\bigg)\,.
	\]
Combining with \eqref{e:zeta.close.to.omega} gives,
for all $\bL,j$ with $\bL(j)=\bt$, the bound
	\beq\label{e:zeta.additive.error.small}
	\bigg|\zeta(\yel\yel\,|\,\bL)-\zeta(\blu\blu\,|\,\bL)\bigg|
	+\bigg|\zeta(\yel\blu\,|\,\bL)-\zeta(\blu\yel\,|\,\bL)\bigg|
	\le O\bigg( \f1{2^{k/10}}\bigg)\,.\eeq
With these estimates in hand, we now turn to estimating the weights in \eqref{e:zeta.lagrangian}.\smallskip

\noindent\bemph{Step 3. Estimation of Lagrangian weights.}
If $\bL$ satisfies $\zeta(\set{\yel\blu,\blu\yel}\,|\,\bL)\ge 2^{-k/15}$, then it follows by combining with \eqref{e:zeta.additive.error.small} that
	\[
	\bigg|\f{\zeta(\yel\blu\,|\,\bL)}{\zeta(\blu\yel\,|\,\bL)}-1\bigg|
	\le
	\f{|\zeta(\yel\blu\,|\,\bL)-\zeta(\blu\yel\,|\,\bL)|
	}{
	\min\set{
	\zeta(\yel\blu\,|\,\bL),\zeta(\blu\yel\,|\,\bL)}}
	\stackrel{\eqref{e:zeta.additive.error.small}}{\le}
	O\bigg(\f{2^{-k/10}}{2^{-k/15}}\bigg)
	=O\bigg( \f1{2^{k/30}}\bigg)\,.
	\]
On the other hand, from the Lagrangian solution \eqref{e:zeta.lagrangian}, we have
	\[
	\f{\zeta(\yel\blu\,|\,\bL)}{\zeta(\blu\yel\,|\,\bL)}
	= \f{\exp\{\beta(\bL)+\gamma+2(\beta_{\bL})^2 \}}
	{\exp\{\beta(\bL)+\gamma+2(\beta_{\bL})^1 \}}
	= \exp\bigg\{ 
	2\Big[
	(\beta_{\bL})^2-(\beta_{\bL})^1\Big]\bigg\}\,.
	\]
Comparing the last two displays, we see that
	\beq\label{e:big.zeta.implies.small.beta}
	\max\bigg\{
	\Big|(\beta_{\bL})^2-(\beta_{\bL})^1\Big| 
	:\bL\ni\bt \textup{ and }\zeta(\set{\yel\blu,\blu\yel}\,|\,\bL)
		\ge \f1{2^{k/15}}
	\bigg\}
	\le
	O\bigg( \f1{2^{k/30}}\bigg)\,.\eeq
An entirely similar argument gives
	\beq\label{e:big.zeta.implies.small.beta.bb}
	\max\bigg\{
	\Big|(\beta_{\bL})^2+(\beta_{\bL})^1\Big| 
	:\bL\ni\bt \textup{ and }\zeta(\set{\yel\yel,\blu\blu}\,|\,\bL)
		\ge \f1{2^{k/15}}\bigg\}
	\le
	O\bigg( \f1{2^{k/30}}\bigg)\,.\eeq
Now suppose $\bt$ is such that $\pi_{\bt}(\set{\yel\blu,\blu\yel})\ge \chi$
for some $2^{-k/15}\le\chi\le 1/2$. It implies that for \bemph{some} clause type $\bL\ni\bt$, we have
$\zeta(\set{\yel\blu,\blu\yel}\,|\,\bL) \ge \chi$. For this particular $\bL$, 
writing $\ch$ for the hyperbolic cosine function, we have
	\begin{align}\nonumber
	\chi &\le \f{\chi}{1-\chi}
	\le
	\f{\zeta(\set{\yel\blu,\blu\yel}\,|\,\bL)}
	{\zeta(\set{\yel\yel,\blu\blu}\,|\,\bL)}
	= e^\gamma
	\f{\exp\{2 (\beta_{\bL})^2\}+\exp\{ 2(\beta_{\bL})^1\}}
	{1 + \exp\{2 (\beta_{\bL})^1+2(\beta_{\bL})^2\}}
	= e^\gamma
	\f{\ch((\beta_{\bL})^1-(\beta_{\bL})^2) }
	{\ch((\beta_{\bL})^1+(\beta_{\bL})^2) }\\
	&\stackrel{\eqref{e:big.zeta.implies.small.beta}}{\le}
	\f{ e^\gamma
	[1 +O(2^{-k/30})]
	}{\ch((\beta_{\bL})^1+(\beta_{\bL})^2) }
	\le e^\gamma
	\bigg( 1 + O\bigg(\f1{2^{k/30}}\bigg)\bigg)
	\,.\label{e:gamma.chi.lbd}
	\end{align}
Next, for \bemph{every} clause type $\bL$ such that $\bL\ni\bt$,
since we assumed $\chi\le1/2$,
at least one of the two quantities
$\zeta(\set{\yel\blu,\blu\yel}\,|\,\bL)$ and
$\zeta(\set{\yel\yel,\blu\blu}\,|\,\bL)$ must be $\ge\chi$.
If $\zeta(\set{\yel\yel,\blu\blu}\,|\,\bL)\ge\chi$, then
	\begin{align*}
	\f{\zeta(\set{\yel\blu,\blu\yel}\,|\,\bL)}
		{\zeta(\set{\yel\yel,\blu\blu}\,|\,\bL)}
	&=e^\gamma
	\f{\ch((\beta_{\bL})^1-(\beta_{\bL})^2) }
	{\ch((\beta_{\bL})^1+(\beta_{\bL})^2) }
	\ge 
	\f{e^\gamma}{\ch((\beta_{\bL})^1+(\beta_{\bL})^2) }\\
	&\stackrel{\eqref{e:big.zeta.implies.small.beta.bb}}{\ge}
	\f{e^\gamma}{1+O(2^{-k/30})}
	\stackrel{\eqref{e:gamma.chi.lbd}}{\ge}
	\chi\bigg( 1 - O\bigg(\f1{2^{k/30}}\bigg)\bigg)
	\,.
	\end{align*}
Combining the two cases
$\zeta(\set{\yel\blu,\blu\yel}\,|\,\bL)\ge\chi$ and
$\zeta(\set{\yel\yel,\blu\blu}\,|\,\bL)\ge\chi$ gives
	\[
	\zeta(\set{\yel\blu,\blu\yel}\,|\,\bL)
	\ge
	\min\bigg\{
	\chi, \f{\chi[1 -O(2^{-k/30})]}{1
		+\chi[1 -O(2^{-k/30})]}
	\bigg\}
	\ge \f{2\chi}{3}
	\bigg( 1- O\bigg(\f1{2^{k/30}}\bigg)\bigg)\,.
	\]
Recalling \eqref{e:zeta.close.to.omega} again, this proves that 
if $\pi_{\bt}(\sigma)\ge 2^{-k/15}$ for $\sigma\in\set{\yel\blu,\blu\yel}$,
then
	\[
	\omega_{\bL,j}(\sigma)
	\ge \f{\chi}{4}
	\ge \f{\pi_{\bt}(\sigma)}{8}\,,
	\]
where the last bound follows by taking
$\chi=\min\set{\pi_{\bt}(\sigma),1/2}$.
The analogous result for
$\sigma\in\set{\yel\yel,\blu\blu}$ by a symmetric argument, and this concludes the proof.
\end{proof}
\end{lem}

\begin{lem}\label{l:lotsOfAsDiverse}
Let $a$ be a clause of type $\bL$ which is both nice
(meaning $\bt$ is nice for all $\bt\in\bL$) and diverse (Definition~\ref{d:diverse}). Given $\omega\in\bm{I}_0$, let $\hbh_{\bL}$ be the probability measure on valid (pair) colorings $\usi_{\delta a}$ which maximizes entropy subject to edge marginals $(\omega_{\bL,j})_j$. Recall from Definition~\ref{d:white.violet} that
each $\usi_{\delta a}$ maps to a configuration
$\uvsi_{\delta a}\in(\mathscr{S}^2)^{\delta a}$. Let
	{\setlength{\jot}{0pt}\begin{align*}
	\SPIN{A}
	&\equiv \{\usi_{\delta a}
	: \sigma_e\ne\red\red \textup{ for all }e\in\delta a\}\,,\\
	\SPIN{W}^i
	&\equiv \{\usi_{\delta a}
	: (\varsigma_e)^i
	=\whi \textup{ for all }e\in\delta a\}\,.
	\end{align*}}%
There exists an absolute constant $\EPSONE>0$ such that
	\beq\label{e:lotsOfAsDiverse.conditional}
	1-\hbh_{\bL}
	\Big(\SPIN{W}^1 \cup \SPIN{W}^2
	\,\Big|\, \SPIN{A}\Big)
	\le \f1{2^{k(1+\EPSONE)}}\,.
	\eeq
If in addition $\bL$ is light (Definition~\ref{d:heavy}), then it follows immediately that
	\beq\label{e:lotsOfAsDiverse.uncond}
	1-\hbh_{\bL}\Big(\SPIN{W}^1 \cup \SPIN{W}^2\Big)
	\le \f1{2^{k(1+\EPSONE)}}
	+ \f{k}{2^{k(1+\EPSLIGHT)}}
	\le O\bigg(\f{k}{2^{k(1+\EPSLIGHT)}}\bigg)\,,
	\eeq
where we can assume that $\EPSLIGHT\le\EPSONE$.

\begin{proof} 
It is an immediate consequence of Lemma~\ref{l:indifferent}
that if $\bL$ is both nice and diverse, then
	\beq\label{e:diverse.clause}
	\sum_{j=1}^{k(\bL)}
	\mathbf{1}\bigg\{
		\omega_{\bL,j}\Big(
		\set{\blu\yel,\yel\blu}\Big)
	\ge \f1{33}
	\bigg\}
	\ge \sum_{j=1}^{k(\bL)}
	\mathbf{1}\bigg\{\omega_{\bL,j}(\yel\blu)
		\ge \f1{65}
	\bigg\}
	\ge \f{k}{10}\,.
	\eeq
In this proof, by the same reasoning as in the proof of Lemma~\ref{l:lotsOfAs} (see \eqref{e:bg.given.cyan}), it suffices to take $\sigma$ in the reduced alphabet $\set{\RYC}$. We also abbreviate $\notr\equiv\set{\yel,\cya}$. Throughout the following, $\EPSONE$ denotes a small positive number, whose value may change from one occurrence to the next, but ultimately is taken as an absolute constant. We denote the clause width by $K\equiv k(\bL)\in\set{k-1,k}$.
\smallskip

\noindent\bemph{Step 1. Reduction to entropy maximization for a conditional measure.} As in the statement of the lemma, let $\hbh_{\bL}$ be the optimizer given marginals $\omega$. Note that since $\omega$ is judicious and $\bL$ is nice,
	\[
	\hbh_{\bL}(\SPIN{A}^c)
	\le\sum_{j=1}^K
	\omega_{\bL,j}(\red\red)
	\le\sum_{i=1,2}
	\sum_{j=1}^K
	(\omega_{\bL,j})^i(\red)
	= 2\sum_{j=1}^K \starpi_{\bL(j)}(\red)
	\le O\bigg( \f{k}{2^k}\bigg)\,.
	\]
Let $\mu(\usi)\equiv \hbh_{\bL}(\usi\,|\, \SPIN{A})$
for $\usi\in\set{\RYC}^{2K}$. For each edge $1\le j\le K$, let $\mu_j$ be the marginal of $\mu$ on the $j$-th edge in $\delta a$. Thus $\mu_j$ is a probability measure over $\set{\RYC}^2\setminus\set{\red\red}$, and the preceding estimate implies
	\beq\label{e:cond.on.A.mgl.effect}
	\mu_j(\sigma)
	=\f{\hbh_{\bL}(\sigma_j=\sigma)
		-\hbh_{\bL}(\sigma_j=\sigma;\SPIN{A}^c)}
		{1-\hbh_{\bL}(\SPIN{A}^c)}
	= \begin{cases}
	\omega_{\bL,j}(\sigma) +O(k/2^k)
		&\textup{if $\red[\sigma]=0$,} \\
	\omega_{\bL,j}(\sigma) +O(1/2^k)
		&\textup{if $\red[\sigma]=1$,}
	\end{cases}
	\eeq
where the last bound uses that $\hbh_{\bL}(\sigma_j=\sigma)\le O(1/2^k)$ if $\red[\sigma]=1$. The measure $\mu$ maximizes entropy subject to the edge marginals $\mu_j$. Similarly as in the preceding proofs of this section, we shall construct a sequence of Lagrangian weights $b_{j,t}$ such that the measure
	\beq\label{e:mu.t.as.fn.of.b.t}
	 \mu_t(\usi)
	= \f{ \Ind{\usi\in\SPIN{A}}}{Z_t}
	\prod_{j=1}^K b_{j,t}(\sigma_j)\eeq
converges to the desired $\mu$ as $t\to\infty$, where $Z_t$ is the normalizing constant. Let $\mu_{j,t}$ denote the marginals of the measure $\mu_t$.\smallskip

\noindent\bemph{Step 2. Initalization of Lagrangian weights.} We initialize the construction with
	\beq\label{e:burnin.reweight.power.of.two}
	b_{j,0}(\sigma)
	\equiv
	\f{\mu_j(\sigma)}{1/ (2^{K-1})^{\red[\sigma]}}\,,\quad
	\sigma\in\set{\RYC}^2\setminus\set{\red\red}\,.
	\eeq
We first estimate the marginals $\mu_{j,0}$ of the resulting measure $\mu_0$. Recall that we denote $\notr\equiv\set{\yel,\cya}$. Note that \eqref{e:cond.on.A.mgl.effect}, together with the judicious condition and the assumption that $\bL$ is nice, gives
	\beq\label{e:zeta.u.u.estimates}
	\mu_j(\notr\notr)=1-O\bigg(\f{k}{2^k}\bigg)\,,\quad
	\sum_{\tau\in\set{\yel,\cya}}
	\Bigg\{
	\bigg|\mu_j(\tau\notr)-\f12\bigg|
	+\bigg|\mu_j(\notr\tau)-\f12\bigg|\Bigg\}
	 \le \f1{2^{k\EPSONE}}\,.
	\eeq
Moreover, it follows from the diverse clause condition \eqref{e:burnin.reweight.power.of.two} together with \eqref{e:cond.on.A.mgl.effect} that
	\beq\label{e:diverse.clause.zeta}
	\sum_{j=1}^{k(\bL)}
	\mathbf{1}\bigg\{
	\min\Big\{	 \mu_j(\yel\blu),\mu_j(\blu\yel)\Big\}
	\ge \f1{65}
	\bigg\}
	\ge \f{k}{10}\,.
	\eeq
Without loss of generality we now focus on the marginal on the edge indexed $j=1$: at time $t=0$, 
	\[
	Z_0\mu_{1,0}(\cya\cya)
	= \mu_1(\cya\cya)
	\Bigg\{
	\prod_{\ell=2}^K\mu_\ell(\notr\notr)
	-\prod_{\ell=2}^K\mu_\ell(\notr\yel)
	-\prod_{\ell=2}^K\mu_\ell(\yel\notr)
	+\prod_{\ell=2}^K\mu_\ell(\yel\yel)
	\Bigg\}
	\stackrel{\eqref{e:zeta.u.u.estimates}}{=}
	\mu_1(\cya\cya)\Bigg\{ 1 +O\bigg(\f{k^2}{2^k}\bigg)\Bigg\}\,,
	\]
Similar calculations, again using \eqref{e:zeta.u.u.estimates}, give the analogous estimate for $\sigma\in\set{\yel,\cya}^2$.
Next, for $\sigma=\red\cya$, it follows from \eqref{e:zeta.u.u.estimates}
together with the diversity bound \eqref{e:diverse.clause.zeta} gives
	\begin{align*}
	Z_0\mu_{1,0}(\red\cya)
	&= \Big( \mu_1(\red\cya) 2^{K-1}\Big)
	\Bigg\{\prod_{\ell=2}^K
	\mu_\ell(\yel\notr)
	-\prod_{\ell=2}^K\mu_\ell(\yel\yel)\Bigg\}\\
	&=\Big( \mu_1(\red\cya) 2^{K-1}\Big)
	\f{1+O(k/2^{k\EPSONE})}{2^{K-1}}\Bigg\{1
	-\bigg(1-\f{2}{65}\bigg)^{k/10}
	\Bigg\}
	= \mu_1(\red\cya) 
	\Bigg\{ 1 +O\bigg(\f1{2^{k\EPSONE}}\bigg)\Bigg\}
	\end{align*}
(recalling our convention that $\EPSONE$ can change from one expression to the next, but remains bounded below by a positive absolute constant). A similar calculation (again using
\eqref{e:zeta.u.u.estimates} and \eqref{e:diverse.clause.zeta}) gives
an analogous estimate for $\sigma=\red\yel$. Altogether we conclude
	\beq\label{e:mu.mgl.zeta.base.case}
	\mu_{1,0}(\sigma)
	=\f{\mu_1(\sigma)}{Z_0}
		\Bigg\{ 1 +O\bigg(\f1{2^{k\EPSONE}}\bigg)\Bigg\}
	= \mu_1(\sigma)
		\Bigg\{ 1 +O\bigg(\f1{2^{k\EPSONE}}\bigg)\Bigg\}\,.
	\eeq
for all $\sigma\in\set{\RYC}^2\setminus\set{\red\red}$.\smallskip

\noindent\bemph{Step 3. Iterative analysis of Lagrangian weights.} Suppose at time $t$ that we have weights $b_{j,t}$, which define a measure $\mu_t$ as in \eqref{e:mu.t.as.fn.of.b.t} with edge marginals $\mu_{j,t}$. We define the updated weights at time $t+1$ by
	\beq\label{e:lotsOfAsDiverse.beta.update.rule}
	\f{b_{j,t+1}(\sigma)}{b_{j,t}(\sigma)}
	= \f{\mu_j(\sigma)}{\mu_{j,t}(\sigma)}\,.
	\eeq
We will prove by induction that for all $t\ge0$, all $1\le j\le k(\bL)$, and all $\sigma\in\set{\RYC}^2\setminus\set{\red\red}$,
	\beq\label{e:lotsOfAsDiverse.induct}
	\Bigg|
	\f{b_{j,t+1}(\sigma)}
		{b_{j,t}(\sigma)}-1 \Bigg|
	=\Bigg|
	\f{\mu_j(\sigma)}{\mu_{j,t}(\sigma)}-1
	\Bigg|
	\le \f{ y_{\red[\sigma]}}{2^{k(t+1)\EPSONE/3}}\,,\quad
	\begin{pmatrix}y_0 \\ y_1\end{pmatrix}
	= \f1{2^{3k\EPSONE/2}}
	\begin{pmatrix}1 \\ 2^{k\EPSONE/2} \end{pmatrix}\,,
	\eeq
where the base case $t=0$ is implied by \eqref{e:mu.mgl.zeta.base.case} (adjusting $\EPSONE$ appropriately). Suppose then that \eqref{e:lotsOfAsDiverse.induct} holds up to time $t-1\ge0$, and note it implies that for all $j,\sigma$ we have
	\beq\label{e:lotsOfAsDiverse.induction}
	\Bigg|\log \f{b_{j,t}(\sigma)}
		{\mu_j(\sigma) (2^{K-1})^{\red[\sigma]}}
		\Bigg|
	=\Bigg|\log \f{b_{j,t}(\sigma)}{b_{j,0}(\sigma)}
		\Bigg|
	\le \sum_{s=0}^{t-1}
	\Bigg|\log \f{b_{j,s+1}(\sigma)}{b_{j,s}(\sigma)}
		\Bigg|
	\stackrel{\eqref{e:lotsOfAsDiverse.induct}}{\le}
		O\bigg( \f1{2^{k\EPSONE}}\bigg)\,.
	\eeq
To analyze the update \eqref{e:lotsOfAsDiverse.beta.update.rule}, let us focus on the first two edges in the clause, and note that
	\beq\label{e:derivative.is.covariance}
	\f{\pd \mu_{1,t}(\sigma)}
		{\pd \log b_{2,t}(\sigma')}
	=\Cov_{\mu_t}\bigg(
	\Ind{\sigma_1=\sigma},
	\Ind{\sigma_2=\sigma'}
	\bigg)
	\eeq
for all $\sigma,\sigma'\in\set{\RYC}^2\setminus\set{\red\red}$ (by direct calculation). We therefore define
	\[
	C_t(\sigma,\sigma')
	\equiv
	\f{Z_t \mu_t(\sigma_1=\sigma,\sigma_2=\sigma')}
		{\mu_{j=1}(\sigma)
		\mu_{j'=2}(\sigma')}\,,
	\]
and proceed to estimate this quantity. In the simplest case $\sigma=\sigma'=\cya\cya$, we have
	\[
	C_t(\cya\cya,\cya\cya)
	=\f{b_{1,t}(\cya\cya)
	b_{2,t}(\cya\cya)}
	{\mu_1(\cya\cya)
	\mu_2(\cya\cya)}
	\prod_{j=3}^K b_{j,t}(\notr\notr)
	\stackrel{\eqref{e:lotsOfAsDiverse.induction}}{=}
	\Bigg\{
	\prod_{j=3}^K \mu_j(\notr\notr)
	\Bigg\}
	\Bigg\{1 + O\bigg(\f1{2^{k\EPSONE}}\bigg)\Bigg\}
	\stackrel{\eqref{e:zeta.u.u.estimates}}{=}
	1 + O\bigg(\f1{2^{k\EPSONE}}\bigg)\,.\]
A similar estimate holds for all cases where
$\red[\sigma]+\red[\sigma']=0$. Next,
	\[
	C_t(\red\cya,\yel\cya)
	=\f{ b_{1,t}(\red\cya)b_{2,t}(\yel\cya)}
		{\mu_1(\red\cya)\mu_2(\yel\cya)}
	\prod_{j=3}^K b_{j,t}(\yel\notr)
	\stackrel{\eqref{e:lotsOfAsDiverse.induction}}{=}
	\Bigg\{ 2^{K-1} \prod_{j=3}^K \mu_j(\yel\notr)
	\Bigg\}
	\Bigg\{1 + O\bigg(\f1{2^{k\EPSONE}}\bigg)\Bigg\}
	\stackrel{\eqref{e:zeta.u.u.estimates}}{=}
	2 \Bigg\{1 + O\bigg(\f1{2^{k\EPSONE}}\bigg)\Bigg\}\,,
	\]
and a similar estimate holds for $C_t(\sigma,\sigma')$ in all cases where $\red[\sigma]+\red[\sigma']=1$. Finally,
	\[
	C_t(\red\yel,\yel\red)
	= \f{b_{1,t}(\red\yel) b_{2,t}(\yel\red)}
		{\mu_1(\red\yel)\mu_2(\yel\red)}
	\prod_{j=3}^K b_{j,t}(\yel\yel)
	\stackrel{\eqref{e:lotsOfAsDiverse.induction}}{=}
	(2^{K-1})^2
	\prod_{j=3}^K b_{j,t}(\yel\yel)
	\Bigg\{1 + O\bigg(\f1{2^{k\EPSONE}}\bigg)\Bigg\}
	\le \f{2^k}{2^{k\EPSONE}}\,,
	\]
where the last bound uses the diversity bound \eqref{e:diverse.clause.zeta}. Combining with \eqref{e:derivative.is.covariance} gives
	\begin{align*}
	\Bigg|\f{\pd \mu_{1,t}(\sigma) /
		\pd \log b_{2,t}(\sigma')}
		{\mu_{1,t}(\sigma)
		\mu_{2,t}(\sigma')}
		\Bigg|
	&=\Bigg| \overbrace{\Bigg\{
		\f{Z_t\mu_t(\sigma_1=\sigma,\sigma_2=\sigma')}
		{\mu_1(\sigma)\mu_2(\sigma')}
		\Bigg\}}^{C_t(\sigma,\sigma')}
	\overbrace{\Bigg\{
	\f{\mu_1(\sigma)\mu_2(\sigma')}
		{Z_t \mu_{1,t}(\sigma)
		\mu_{2,t}(\sigma')}
		\Bigg\}}^{1 + O(2^{-k\EPSONE})}
		{} -1 \Bigg| \\
	&\le O(1) \begin{cases}
	2^{-k\EPSONE} & \textup{if 
		$\red[\sigma]+\red[\sigma']=0$,}\\
	1 & \textup{if $\red[\sigma]+\red[\sigma']=1$,}\\
	2^{k(1-\EPSONE)} & \textup{if 
		$\red[\sigma]=\red[\sigma']=1$.}
	\end{cases}
	\end{align*}
Rewriting the above in a more convenient form gives
	\beq\label{e:partial.deriv.covar.bound}
	\f1{\mu_1(\sigma)}
	\Bigg|\f{\pd \mu_{1,t}(\sigma)}
		{\pd\log b_{2,t}(\sigma')}
		\Bigg|
	\le O\Big( A_{\red[\sigma],\red[\sigma']}
		\Big)\,,\quad
	A = \begin{pmatrix} A_{0,0} & A_{0,1}\\
		A_{1,0} & A_{1,1} \end{pmatrix}
	= \begin{pmatrix} 2^{-k\EPSONE} & 2^{-k} \\
		1 & 2^{-k\EPSONE} \end{pmatrix}\,.\eeq
Now recall the definition \eqref{e:mu.t.as.fn.of.b.t}
of $\mu_t$, and define
	\[
	\tilde{\mu}_t(\usi)
	= \f{ \Ind{\usi\in\SPIN{A}}}{Z_t}
	b_{1,t}(\sigma_1)
	\prod_{j=2}^K b_{j,t-1}(\sigma_j)\,.
	\]
It follows from the update rule \eqref{e:lotsOfAsDiverse.beta.update.rule} that the marginal of $\tilde{\mu}_t$ on the first edge is precisely $\mu_1$. Consequently,
	\begin{align*}
	\Bigg|
	\f{b_{1,t+1}(\sigma)}
		{b_{1,t}(\sigma)}-1 \Bigg|
	&\stackrel{\eqref{e:lotsOfAsDiverse.beta.update.rule}}{=}\Bigg|
	\f{\tilde{\mu}_t(\sigma_1=\sigma)}
		{\mu_t(\sigma_1=\sigma)}-1
	\Bigg|
	\le \f{O(1)}{\mu_1(\sigma)}
	\sum_{j=2}^K\sum_{\sigma'}
	\Bigg| \f{\pd\mu_{1,t}(\sigma)}
		{\pd \log b_{j,t}(\sigma')}\Bigg|
	\cdot
	\Bigg|\log \f{b_{j,t}(\sigma')}
		{b_{j,t-1}(\sigma')}\Bigg|\\
	&\stackrel{\eqref{e:partial.deriv.covar.bound}}{\le}
	\sum_{\sigma'} A_{\red[\sigma],\red[\sigma']}
	\sum_{j=2}^K\sum_{\sigma'}
	\Bigg|\log \f{b_{j,t}(\sigma')}
		{b_{j,t-1}(\sigma')}\Bigg|
	\stackrel{\eqref{e:lotsOfAsDiverse.induct}}{\le}
	\f{k}{2^{kt\EPSONE/3}} \f1{2^{3k\EPSONE/2}}
	\sum_{\sigma'}
	A_{\red[\sigma],\red[\sigma']}
	(2^{k\EPSONE/2})^{\red[\sigma]}\\
	&\stackrel{\eqref{e:partial.deriv.covar.bound}}{\le}
	\f1{2^{k(t+1)\EPSONE/3}}
	\f{(2^{k\EPSONE/2})^{\red[\sigma]}}{2^{3k\EPSONE/2}} \,.
	\end{align*}
This verifies the induction and proves
\eqref{e:lotsOfAsDiverse.induct} for all $t\ge0$.
Taking $t\to\infty$ gives the weights $b=b_\infty$ for the optimal measure $\mu$. We can then use this to estimate
	\begin{align*}
	1-\mu\Big(\SPIN{W}^1 \cup \SPIN{W}^2\Big)
	&\le
	\mu\bigg( \sum_{j=1}^K
	\mathbf{1}\Big\{\sigma_j
		\in\set{\cya\cya,\cya\yel,\yel\cya}\Big\} \le4
	\bigg)
	\le
	\mu\bigg( \sum_{j=1}^K
	\Ind{\sigma_j=\yel\yel} \ge K-6\bigg)
	\\
	&\le \sum_{|S| \le 6} \prod_{j\in[K]\setminus S}
			 b_j(\yel\yel)
	\le \f1{2^{k(1+\EPSONE)}}\,,
	\end{align*}
where the last bound follows from \eqref{e:lotsOfAsDiverse.induct} together with the diversity bound \eqref{e:diverse.clause.zeta}. This implies the first assertion
\eqref{e:lotsOfAsDiverse.conditional} of the lemma. The second assertion 
\eqref{e:lotsOfAsDiverse.uncond}
(removing the conditioning on $\SPIN{A}$)
follows trivially, since the assumption that $\bL$ is light means (by Definition~\ref{d:heavy}) that
$\omega_{\bL,j}(\red\red)\le 2^{-k(1+\EPSLIGHT)}$
for all $j$, and so a union bound gives
	\[\hbh_{\bL}(\SPIN{A}^c )\le
	\sum_{j=1}^{k(\bL)} \omega_{\bL,j}(\red\red)
	\le \f{k}{ 2^{k(1+\EPSLIGHT)}}\,.
	\]
It follows by combining with \eqref{e:lotsOfAsDiverse.conditional} that
	\[
	\hbh_{\bL}\bigg(
		(\SPIN{W}^1 \cup \SPIN{W}^2)^c\bigg)
	\le
	\hbh_{\bL}(\SPIN{A}^c )
	+\hbh_{\bL}\bigg(
	(\SPIN{W}^1 \cup \SPIN{W}^2)^c \,\bigg|\,\SPIN{A}^c
	\bigg)
	\le \f1{2^{k(1+\EPSONE)}}
	+ \f{k}{2^{k(1+\EPSLIGHT)}}\,,
	\]
as claimed.\end{proof}
\end{lem}

\subsection{Entropy maximization around forced variables} 

In this subsection we show that under certain conditions, edges with $\sigma^i=\SPIN{red}$ are likely to have $\ups^i=\SPIN{violet}$.
The precise statements are given in Lemmas~\ref{l:rrUnforced}~and~\ref{l:ryUnforced} below. In later subsections we will prove bounds on $\ups=\SPIN{violet}$ edges, and use the results from this subsection to deduce bounds on $\sigma=\SPIN{red}$ edges. 

In \S\ref{ss:white.diverse} we worked with configurations $\uvsi\in\set{\RYC,\whi}^{2E}$ --- we noted that $\usi_{\delta a}$ determines $\uvsi_{\delta a}$ for each clause $a$, so to estimate the joint distribution of $(\sigma,\varsigma)$ under $\bom_{\bL,j}$ it suffices to consider only the clause measures $\hbh_{\bL}$. In this subsection, however, the aim is to prove estimates concerning the configurations $\vups\in\set{\RYGB,\mred}^{2E}$, which can no longer be determined from the clause measures $\hbh_{\bL}$. Instead, $\usi_{\delta v}$ determines $\vups_{\delta v}$ for each variable $v$. However, the joint distribution of $(\sigma,\ups)$ under $\bom_{\bL,j}$ cannot be inferred from only the variables measures $\dbh_{\bT}$, since those do not account for the distribution of clause types.

To resolve these issues, we again use the device of augmenting the alphabet with the clause type, similarly as in Definition~\ref{d:judicious.augmented.alphabet}. On a graph $\GG=(V,F,E)$ let $(\usi,\uL)$ denote an augmented pair coloring. Let $\Omega$ denote the vertex empirical measure for the augmented coloring:
for each variable type $\bT$ we let
$\Omega_{\bT}$ be the empirical measure of augmented colorings $(\usi,\uL)_{\delta v}$ on variables $v$ of type $\bT$.
Likewise we define $\Omega_{\bt}$ for edge types $\bt$, and $\Omega_{\bL}$ for clause types $\bL$. 
The combinatorial calculation \eqref{e:config.model.pair.version} also implies that the contribution of $\Omega$ to the second moment is
	\beq\label{e:Psi.of.Omega}
	\E_{\DD} \ZZ^2[\Omega]
	= \underbrace{
	\Bigg\{
	\prod_{\bT} 
	\binom{n_{\bT}}{n_{\bT}\Omega_{\bT}}
	\prod_{\bL}
	\binom{m_{\bL}}{m_{\bL}\Omega_{\bL}}
	\Bigg\}
	}_{\substack{
	\text{number of colorings}\\
	\text{prior to matching}
	}}
	\underbrace{\Bigg\{
	\prod_{\bm{t}} \binom{n_{\bm{t}}}
		{ n_{\bm{t}} \Omega_{\bm{t}} }
		\Bigg\}^{-1}
		}_{\substack{\text{probability of matching}\\
			\text{to respect colorings}}}
	=
	\f{\exp\{n\bm{\Phi}_{\DD,2}(\Omega)\}}
		{n^{O(1)}}\,,\eeq
We write $\Omega\sim\nu$ if $\Omega$ is consistent with $\nu =(\dbh,\hbh)$: in this case, 
	\beq\label{e:Omega.bL.trivial}
	\Omega_{\bL}(\usi,\uL)
	=\Bigg\{ \prod_{j=1}^{k(\bL)}
		\Ind{\bL_j=\bL}\Bigg\}
	\hbh_{\bL}(\usi)\,\eeq
If $\Omega\sim\nu$ and $\nu\sim\omega$, then we must have $\Omega_{\bt}=(\AUGMENT_{\DD}(\omega))_{\bt}$
in the notation of \eqref{e:pi.DD.first.appearance}. Thus we see that only the variable measures $\Omega_{\bT}$ carry more information than $\nu$.

\begin{lem}\label{l:rrUnforced} Let $\EPSTWO$ be a small absolute constant. Fix $x\in\set{\plus,\minus}$, and let $v$ be a variable of type $\bT$. Suppose we are given $\omega\in\bm{I}_0$ with marginal $\pi$ such that for both $i=1,2$ we have 
	\beq\label{e:assumption.rr.avg}
	\gamma^i\equiv
	\sum_{e\in\delta v(x)}
	\pi_e\Big(\sigma^i=\red
		\,\Big|\,\sigma\in\set{\red,\blu}^2
		\Big)
	\ge k\EPSTWO\,,\eeq
where $\pi_e\equiv\pi_{\bt}$ for an edge $e$ of type $\bt$. Let $\Omega$ be the maximizer of \eqref{e:Psi.of.Omega} that is consistent with $\omega$, and let $\bom_{\bL,j}$ be the resulting joint distribution of $(\sigma,\ups)$. Then there exists an absolute constant $\EPSTHREE$
(depending only on $\EPSTWO$) such that for every edge $e\in\delta v(x)$ we have
	\[\left.
	\begin{array}{r}
	\bom_{\bL,j}( 
		\ups=\mred\mred \,|\, \sigma=\red\red)\\
	\bom_{\bL,j}( \ups^1=\mred
		\,|\, \sigma=\red\blu )\\
	\bom_{\bL,j}(
		\ups^2=\mred \,|\,\sigma=\blu\red )\\
	\end{array} \hspace{-4pt}
	\right\}
	\ge 1-\f1{2^{k\EPSTHREE}}
	\]
for $j=j(\bt)$ and all clause types $\bL$ such that 
$\bL(j)=\bt$.

\begin{proof} Throughout the proof, $\EPSTHREE$ denotes a small positive number, whose value may change from one occurrence to the next, but ultimately is taken as an absolute constant that depends only on $\EPSTWO$.\smallskip

\noindent\bemph{Step 1. Reduction to entropy maximization for a conditional measure.}
Note from \eqref{e:Psi.of.Omega} that if we fix
$\Omega_{\bt}$, then we obtain separate entropy maximization problems over $\Omega_{\bT}$ and $\Omega_{\bL}$. Let us therefore fix $\Omega_{\bt}=(\AUGMENT_{\DD}(\omega))_{\bt}$, and consider the optimization over $\Omega_{\bT}$. Let $v$ denote a variable of type $\bT$, and write $\Omega_v\equiv\Omega_{\bT}$.
By the method of Lagrange multipliers, the optimal $\Omega_v$ must take the form
	\[
	\Omega_v\Big( (\usi,\bL)_{\delta v}\Big)
	= \f{\varphi_v(\usi_{\delta v})}{\dbz} 
		\prod_{e\in\delta v}
		\hq_e(\sigma_e,\bL_e)\,,\]
where each $\hq_e$ is a probability measure over elements $(\sigma,\bL)$, chosen such that $\Omega_v$ has marginals $\Omega_e\equiv\Omega_{\bt}$ for $e\in\delta v$ and $\bt\equiv\bt_e$. The marginal of $\Omega_v$ on $(\usi_{\delta v},\bL_e)$ is
	\[
	\Omega_v(\usi_{\delta v},\bL_e)
	= \f{\varphi_v(\usi_{\delta v})}{\dbz}
	\hq_e(\sigma_e,\bL_e)
	\prod_{e'\in\delta v\setminus e}
	\avhq_{e'}(\sigma_{e'})\,,
	\]
where $\avhq_{e'}$ denotes the marginal of $\hq_{e'}$ on $\sigma_{e'}$ alone. It follows that
	\beq\label{e:var.spins.types.cond.indep}
	\Omega_{\bT}\bigg(\usi_{\delta v\setminus e}
		\,\bigg|\, (\sigma_e,\bL_e)\bigg)
	= \dbh_{\bT}
		\Big(\usi_{\delta v\setminus e}
		\,\Big|\,\sigma_e\Big)
	\cong 
	\varphi_v(\usi_{\delta v})
	\prod_{e'\in\delta v\setminus e}
	\avhq_{e'}(\sigma_{e'}).
	\eeq
Thus, for the purposes of this lemma, it suffices to estimate only $\dbh_{\bT}(\cdot\,|\,\sigma_e)$. To this end, assume without loss $x=\plus$, and consider the event $x_v=\plus\plus$. In this case 
$\sigma_e\in\set{\red,\blu}^2$ for all $e\in\delta v(\plus)$, and
$\sigma_e=\yel\yel$ for all $e\in\delta v(\minus)$. Let
	\beq\label{e:def.mu.cond.on.plus.plus}
	\mu(\usi)
	\equiv
	\dbh_{\bT}\bigg(\usi_{\delta v(\plus)}=\usi
		\,\bigg|\,x_v=\plus\plus\bigg)\,.\eeq
Denote $D \equiv |\delta v(\plus)|$.
Then $\mu$ is a probability measure over colorings $\usi\in\set{\red,\blu}^{2D}$ with at least one $\SPIN{red}$ spin in each copy $i=1,2$. Moreover, $\mu$ must maximize entropy subject to its marginals 
$\mu_\ell$. It 
follows by the method of Lagrange multipliers that there exist probability measures $h_\ell$
over $\set{\red,\blu}^2$ such that
	\beq\label{e:rrUnforced.mu.q}
	\mu(\usi)
	\cong
	\Bigg\{
	\prod_{i=1,2}\Ind{\red[\usi^i]\ge1}\Bigg\}
	\Bigg\{
	\prod_{\ell=1}^D h_\ell(\sigma_\ell)
	\Bigg\}\,.
	\eeq
We now turn to the construction and estimation of the $h_\ell$.\smallskip

\noindent\bemph{Step 2. Construction of Lagrangian weights.} As usual, we will construct $h_{\ell,t}\to h_\ell$ in the limit $t\to\infty$. We initialize the construction with $h_{\ell,0}\equiv\mu_\ell$. Let $\mu_t$
be defined by \eqref{e:rrUnforced.mu.q} with $h_{\ell,t}$ in place of $h_\ell$, and let $\mu_{\ell,t}$ denote the marginal of $\mu_t$ on the $\ell$-th edge in $\delta v(\plus)$. Similarly
to \eqref{e:q.times.prob.of.remainder} we have
	\beq\label{e:def.xi.ell.t}
	\mu_{\ell,t}(\sigma)
	\cong
	h_{\ell,t}(\sigma) \overbrace{
	\sum_{\usi_{-\ell}}
	\Bigg\{
	\prod_{i=1,2}\Ind{\red[\usi^i]\ge1}\Bigg\}
	\prod_{\ell'\ne\ell}
	h_{\ell',t}(\sigma_{\ell'})
	}^{\textup{denote this }\xi_{\ell,t}(\sigma)}\,.
	\eeq
For the remainder of the proof, we assume without loss that $\ell=1$. Then at $t=0$ we have
	\beq
	0\le 1-\xi_{1,0}(\red\blu)
	= \prod_{\ell=2}^D \bigg\{
		1- \mu_\ell(\pur\red)
		\bigg\}
	\le \exp\Bigg\{ - \sum_{\ell=2}^D
		\mu_\ell(\pur\red)\Bigg\}
	\stackrel{\eqref{e:def.mu.cond.on.plus.plus}}{\le}
		\f{e}{\exp(\gamma^2)}
	\stackrel{\eqref{e:assumption.rr.avg}}
		{\le} \f1{2^{k\EPSTHREE}}\,.
	\label{e:mostly.const.over.rb}\eeq
By similar calculations, $0\le 1-\xi_{1,0}(\sigma) \le 2^{-k\EPSTHREE}$ for all $\sigma\in\set{\red,\blu}^2$. (In fact $\xi_{1,0}(\red\red)=1$.) Given $h_{\ell,t}$ we define $h_{\ell,t+1}$ to be the probability measure such that
	\beq\label{e:update.rule.h}
	h_{\ell,t+1}(\sigma)
	\cong \f{\mu_\ell(\sigma)}{\xi_{\ell,t}(\sigma)}
	\stackrel{\eqref{e:def.xi.ell.t}}{\cong}
		h_{\ell,t}(\sigma)
		\f{\mu_\ell(\sigma)}{\mu_{\ell,t}(\sigma)}
	\eeq
for all $\sigma\in\set{\red,\blu}^2$.\smallskip

\noindent\bemph{Step 3. Estimation of Lagrangian weights.} Suppose inductively that
	\beq\label{e:rr.induct}
	\bigg\|\f{\mu_{\ell,t}}{\mu_\ell}-1 \bigg\|_\infty
	\le \f1{2^{k\EPSTHREE(t+1)/3}}\eeq
for all $1\le\ell\le D$, where the base case $t=0$ follows from the above bounds (see \eqref{e:mostly.const.over.rb}) on $\xi_{\ell,0}$. 
If we assume \eqref{e:rr.induct} holds up to $t$, then taking a telescoping sum gives
	\begin{align}\nonumber
	&\Big|\xi_{1,t+1}(\red\blu)
	-\xi_{1,t}(\red\blu)\Big|
	= \Bigg| \prod_{j=2}^D
		\Big\{1-h_{\ell,t+1}(\pur\red)\Big\}
	-\prod_{j=2}^D
		\Big\{ 1-h_{\ell,t}(\pur\red)\Big\}
	\Bigg| \\ \nonumber
	&\qquad\le \sum_{\ell=2}^D
	h_{\ell,t}(\pur\red)
	\Bigg| \f{h_{\ell,t+1}(\pur\red)}
		{h_{\ell,t}(\pur\red)}-1\Bigg|
	\prod_{j=2}^{\ell-1}
	\bigg\{1-h_{j,t+1}(\pur\red)\bigg\}
	\prod_{j=\ell+1}^D
	\bigg\{1-h_{j,t}(\pur\red)\bigg\}\\ \nonumber
	&\qquad\le O(1)\Bigg\{
	\max_{1\le\ell\le D}
	\bigg| \f{h_{\ell,t+1}(\pur\red)}
		{h_{\ell,t}(\pur\red)}-1
	\bigg|\Bigg\}\Bigg\{
	\sum_{\ell=1}^D h_{\ell,t}(\pur\red)
	\Bigg\}\Bigg/
	\exp\Bigg\{
	\sum_{\ell=1}^D h_{\ell,t}(\pur\red)
	\Bigg\} \\
	&\qquad\stackrel{\eqref{e:rr.induct}}{\le}
	O\Bigg( \f{\gamma^2}{\exp(\gamma^2)}\Bigg)
	\Bigg\{
	\max_{1\le\ell\le D}
	\bigg| \f{h_{\ell,t+1}(\pur\red)}
		{h_{\ell,t}(\pur\red)}-1
	\bigg|\Bigg\}
	\stackrel{\eqref{e:update.rule.h}}{=}
	O\bigg( \f{\gamma^2}{\exp(\gamma^2)}\bigg)
	\max_{1\le\ell\le D}
	\bigg| \f{\mu_\ell(\pur\red)}{\mu_{\ell,t}(\pur\red)} -1
	\bigg|
	\stackrel{\eqref{e:rr.induct}}{\le}
	\f{2^{-k\EPSTHREE}}{2^{k\EPSTHREE(t+1)/3}}\,.
	\label{e:d.xi.d.h}
	\end{align}
A similar estimate holds for
$\xi_{1,t+1}(\sigma)-\xi_{1,t}(\sigma)$ for all $\sigma\in\set{\red,\blu}^2$.
Next note that if we let $\tilde{\mu}_{t+1}$ be defined by
\eqref{e:rrUnforced.mu.q} with weights 
$h_{1,t+1}$ and $h_{\ell,t}$ for $\ell\ge2$, then
the marginal on the first edge is, similarly to \eqref{e:def.xi.ell.t},
	\[
	\tilde{\mu}_{1,t+1}(\sigma)
	\cong h_{1,t+1}(\sigma)
		\xi_{1,t}(\sigma)
	\stackrel{\eqref{e:update.rule.h}}{\cong}
	\f{\mu_1(\sigma)}{\xi_{1,t}(\sigma)}
		\xi_{1,t}(\sigma)
	=\mu_1(\sigma)\,.
	\]
It follows from this that 
	\begin{align*}
	\bigg\|\f{\mu_{1,t+1}}{\mu_1}
		-1 \bigg\|_\infty
	&=\bigg\|
		\f{\mu_{1,t+1}}{\tilde{\mu}_{1,t+1}}
		-1 \bigg\|_\infty
	\stackrel{\eqref{e:update.rule.h}}{=} \bigg\|
		\f{\xi_{1,t+1}}{\xi_{1,t}}
		-1 \bigg\|_\infty
	\stackrel{\eqref{e:d.xi.d.h}}{\le}
	O\bigg(\f{2^{-k\EPSTHREE}}{2^{k\EPSTHREE(t+1)/3}}\bigg)
	\le \f1{2^{k\EPSTHREE(t+2)/3}}\,,
	\end{align*}
which verifies the induction. We then take $t\to\infty$ to obtain the limiting weights $h_\ell\equiv h_{\ell,\infty}$ that define the optimal measure $\mu$ by \eqref{e:rrUnforced.mu.q}. Suppose $e$ corresponds to the first edge in $\delta v(\plus)$. Then, for all $\sigma_e\in\set{\red,\blu}^2$,
	\[
	\dbh_{\bT}\bigg(
	\sum_{e'\in\delta v(\plus)\setminus e}
	\Ind{ (\sigma_{e'})^2=\red} =0
	\,\bigg|\,\sigma_e
	\bigg)
	\stackrel{\eqref{e:def.mu.cond.on.plus.plus}}{=}
	\mu\bigg(
	\usi_{-1}\in\set{\pur\blu}^D \,\bigg|\,\sigma_1
	\bigg)
	\stackrel{\eqref{e:rr.induct}}{\le}
	O\Bigg(
	\prod_{\ell=2}^D
		\bigg\{ 1-h_\ell(\pur\red)\bigg\}
		\Bigg) \le \f1{2^{k\EPSTHREE}}\,.
	\]
The lemma follows by recalling \eqref{e:var.spins.types.cond.indep}.
\end{proof}
\end{lem}

\begin{lem}[{used only in proof of Lemma~\ref{l:large11Case}}]
\label{l:ryUnforced}
Let $\EPSTWO$ be a small absolute constant. Fix $x\in\set{\plus,\minus}$ and $\tau\in\set{\yel,\grn}$, and let $v$ be a variable of type $\bT$. Suppose we are given $\omega\in\bm{I}_0$ with marginal $\pi$ such that
	\beq\label{e:ryUnforced.condition} 
	\sum_{e\in\delta v(x)}\pi_e\Big(
		\sigma^1=\red\,\Big|\,
	\sigma\in\set{\red\tau,\blu\tau}\Big)
	\ge k\EPSTWO\,,\eeq
where $\pi_e\equiv\pi_{\bt}$ for an edge $e$ of type $\bt$. Let $\Omega$ be the maximizer of \eqref{e:Psi.of.Omega} that is consistent with $\omega$, and let $\bom_{\bL,j}$ be the resulting joint distribution of $(\sigma,\ups)$. Then there exists a positive absolute constant $\EPSTHREE$ (depending only on $\EPSTWO$) such that for every edge $e\in\delta v(x)$ we have
	\[\bom_{\bL,j}\Big(
		\ups^1=\mred
		\,\Big|\,
		\sigma=\red\tau\Big)
		\ge 1-\f1{2^{k\EPSTHREE}}\]
for $j=j(\bt)$ and all clause types $\bL$ such that 
$\bL(j)=\bt$.
The same holds if we exchange the two copies $i=1,2$.

\begin{proof}
The proof is very similar to (but simpler than) that of Lemma~\ref{l:rrUnforced}, and we omit the details.
\end{proof}
\end{lem}

\subsection{Bounds on doubly forced edges}

In this subsection we bound the incidence of edges $e$ which are forced in both coordinates, meaning that $\sigma_e=\red\red$ and $\ups_e\in\set{\red,\mred}^2$. 

\begin{lem}\label{l:ForcedRR}
If the clause type $\bL$ is nice
and diverse (Definition~\ref{d:diverse}), then
	\[
	\bom_{\bL,j}(\ups=\mred\mred)
	\le \f1{2^{k(1+\EPSDIV)}}\]
for all $1\le j\le k(\bL)$, where $\EPSDIV$ is an absolute constant.

\begin{proof} Throughout the proof, $\EPSDIV$ denotes a small positive number, whose value may change from one occurrence to the next, but ultimately is taken as an absolute constant. Fix $\bL,j$, and let $\Xi$ denote the empirical measure of configurations $(\col_e,\usi_{(\delta v\cup\delta a)\setminus e})$ over all $e=(av)$ such that $\bL_a=\bL$ and $j(v;a)=j$. In the proof below we assume without loss of generality that the edge $e$ has label $\lit_e=\plus$.\smallskip

\noindent\bemph{Step 1. Reduction to entropy maximization for a conditional measure.} Let $\SPIN{V}$ denote the subset of all valid configurations $(\col_e,\usi_{(\delta v\cup\delta a)\setminus e})$ that have $\col_e\in\set{\whi,\mred}^2$. Define the conditional measure
	\[
	\mu\Big(\col_e,\usi_{(\delta v\cup\delta a)\setminus e}\Big)
	\equiv
	\Xi\bigg(\col_e,\usi_{(\delta v\cup\delta a)\setminus e}\,\bigg|\, \SPIN{V}\bigg)\,.
	\]
Note the marginals of $\mu$ must be close to those of the original measure $\Xi$, since Lemma~\ref{l:lotsOfAs} gives
	\beq\label{e:mu.close.to.Xi}
	\Xi(\SPIN{V})
	\ge
	\bom_{\bL,j}(\varsigma=\whi\whi)
	\ge 1-O\bigg(\f{k^2}{2^k}\bigg)\,.\eeq
The measure $\mu$ maximizes entropy subject to the marginal distributions of $\sigma_{e'}$ for $e'\ne e$,
as well as of $(\col_e)^i$ for $i=1,2$. By the method of Lagrange multipliers, we can express
	\beq\label{e:mu.two.sided.tree}
	\mu\Big(\col_e,\usi_{(\delta v\cup\delta a)\setminus e}\Big)
	\cong 
	\Bigg\{\prod_{i=1,2} (b_e)^i((\col_e)^i)\Bigg\}
	\Bigg\{
	\prod_{e'\in(\delta v \cup \delta a)\setminus e}
	b_{e'}(\sigma_{e'})\Bigg\}
	\mathbf{1}_{\SPIN{V}}\,,
	\eeq
where we fix $(b_e)^i(\whi)\equiv1$ for both $i=1,2$.
For $x\in\set{\minus,\plus,\free}^2$, let us write
$(\col_e,\usi_{\delta v\setminus e})\sim x$
if $(\col_e,\usi_{\delta v\setminus e})$ is consistent with frozen spin $x_v=x$. Note that for $\col_e\in\set{\whi,\mred}^2$, we have
	\[
	\bigg\{\usi_{\delta v\setminus e}
	: (\col_e,\usi_{\delta v\setminus e})\sim x
	\bigg\}
	= S(x) \equiv \prod_{i=1,2} S^i(x^i)\,,
	\]
where $S^i(x^i)$ \bemph{does not depend on $\col_e$}. For instance, for all $\col_e\in\set{\whi,\mred}^2$,
	\[
	\bigg\{\usi_{\delta v\setminus e}
	: (\col_e,\usi_{\delta v\setminus e})
	\sim \plus\plus
	\bigg\}
	=\prod_{i=1,2}
	\overbrace{\Bigg\{(\usi_{\delta v\setminus e})^i
	: (\usi_{\delta v(\minus)})^i
		\equiv\yel,
	(\usi_{\delta v(\plus)\setminus e})^i
		\equiv \pur,
	\sum_{e'\in \delta v(\plus)\setminus e}
	\Ind{(\sigma_{e'})^i=\red}\ge1
	\Bigg\}}^{S^i(\plus)}\,.
	\]
Consequently, the marginal of \eqref{e:mu.two.sided.tree}
on $(x_v,\col_e,\usi_{\delta a\setminus e})$ can be written as
	\beq\label{e:ForcedRR.mu.of.beta}
	\mu(x_v,\col_e,\usi_{\delta a\setminus e})
	=\f{\mathbf{1}_{\SPIN{V}}}Z
	\Bigg\{\prod_{i=1,2} (b_e)^i((\col_e)^i)\Bigg\}
	\overbrace{
	\Bigg\{\sum_{\usi_{\delta v\setminus e} \in S(x_v)}
	\prod_{e'\in\delta v\setminus e}
	b_{e'}(\sigma_{e'})
	\Bigg\}}^{\textup{denote this }b_v(x_v)}
	\Bigg\{ \prod_{e'\in \delta a\setminus e}
	b_{e'}(\sigma_{e'})\Bigg\}\,,\eeq
where $Z$ is the normalizing constant. We next turn to the construction and estimation of the weights in \eqref{e:ForcedRR.mu.of.beta}.\smallskip

\noindent\bemph{Step 2. Initialization of Lagrangian weights.} We will construct weights $b_t\to b$ in the limit $t\to\infty$. At $t=0$, on the variable $v$ and on the edges $\delta a\setminus e$, we put
	{\setlength{\jot}{0pt}\begin{alignat}{2}
	\nonumber
	b_{v,0}(x)
		&\equiv \mu_v(x)
		\equiv \mu(x_v=x)\quad
		&&\textup{for 
		$x\in\set{\minus,\plus,\free}^2$,} \\
	b_{e',0}(\sigma)
		&\equiv \mu_{e'}(\sigma)
		\equiv \mu( \sigma_{e'}=\sigma )\quad
		&&\textup{for 
		$\sigma\in\set{\yel,\cya}^2$
		and $e'\in\delta a\setminus e$.}
	\label{e:vv.lemma.initialization}
		\end{alignat}}%
Abbreviate $K\equiv k(\bL)$. On the central edge $e=(av)$, for $i=1,2$ we put
	\[(b_{e,0})^i( (\col_e)^i )
		\equiv
		\f{(\mu_e)^i( (\col_e)^i)}
			{ 1/ (2^K)^{\Ind{(\col_e)^i=\mred}} }\,.
		\]
Let $\mu_t$ be defined as $\mu$ in \eqref{e:ForcedRR.mu.of.beta}, but with $b_t$ in place of $b$, and normalizing constant $Z_t$. We begin by estimating the marginals at time $t=0$. Recalling the notation $\notr\equiv\set{\yel,\cya}$, we have
	\[
	Z_0\mu_{e,0}(\whi\whi)
	= (\mu_e)^1(\whi)(\mu_e)^2(\whi)
	\Bigg\{\prod_{e'\in\delta a\setminus e}
	\mu_{e'}(\notr\notr)\Bigg\}
	\Bigg\{ 1 - O\bigg(\f{k^2}{2^k}\bigg)\Bigg\}
	= (\mu_e)^1(\whi)(\mu_e)^2(\whi)
	\Bigg\{ 1 - O\bigg(\f{k^2}{2^k}\bigg)\Bigg\}\,,
	\]
where the estimate uses \eqref{e:mu.close.to.Xi} and the assumption that $\bL$ is nice, which implies
(together with the judicious condition) $\Xi_{e'}(\notr\notr)\ge 1-O(2^{-k})$ for all $e'\in\delta a$. (In the above calculation, the $O(k^2/2^k)$ error comes from the fact that among the configurations $\usi_{\delta a\setminus e}$ with $\sigma_{e'}\in\notr\notr$ for all $e'$, some will not be compatible with $\sigma_e=\whi\whi$ because they will have too few $\cya$ spins.) Next recall that (as noted in the proof of Lemma~\ref{l:lotsOfAsDiverse}) if $\bL$ is both nice and diverse, then Lemma~\ref{l:indifferent} implies \eqref{e:diverse.clause}, which says that $\omega_{\bL,j}(\yel\blu)\ge1/65$ for at least $k/10$ indices $1\le j\le k(\bL)$. Recall also that \eqref{e:mu.close.to.Xi} implies that the marginals of $\mu$ are close to those of the original measure $\Xi$, so \eqref{e:diverse.clause} implies
	\beq\label{e:diverse.clause.consequence}
	\min\Bigg\{
	\sum_{j=1}^{k(\bL)}
	\mathbf{1}\bigg\{
	\mu_j(\yel\blu)\ge\f1{66}
	\bigg\},
	\sum_{j=1}^{k(\bL)}
	\mathbf{1}\bigg\{
	\mu_j(\blu\yel)\ge\f1{66}
	\bigg\}
	\Bigg\} \ge \f{k}{10}\,.\eeq
It follows using \eqref{e:diverse.clause.consequence} that 
	\begin{align*}
	Z_0\mu_{e,0}(\mred\whi)
	&=
	(\mu_e)^1(\mred)(\mu_e)^2(\whi)
	\f{\mu_v(\plus) }{1/2}
	\Bigg\{\prod_{e'\in\delta a\setminus e}
	\f{\mu_{e'}(\yel\notr)}{1/2}
	\Bigg\}
	\Bigg\{ 1 - O\bigg(\f1{2^{k\EPSDIV}}\bigg)\Bigg\}\\
	&=
	(\mu_e)^1(\mred)(\mu_e)^2(\whi)
	\Bigg\{ 1 - O\bigg(\f1{2^{k\EPSDIV}}\bigg)
	\Bigg\}\,,
	\end{align*}
where the factors of $1/2$ come from the $2^K$ term in the definition of $(b_{e,0})^1(\mred)$. Lastly,
	\beq\label{e:ForcedRR.init.pair.mgl.on.edge}
	Z_0\mu_{e,0}(\mred\mred)
	=\overbrace{
	\f{(\mu_e)^1(\mred)}{1/2^K}
	\f{(\mu_e)^2(\mred)}{1/2^K}
	}^{O(1)}
	\overbrace{\mu_v(\plus\plus)
	\Bigg\{\prod_{e'\in\delta a\setminus e}
	\mu_{e'}(\yel\yel)
	\Bigg\}}^{\textup{at most $2^{-k(1+\EPSDIV)}$
		by \eqref{e:diverse.clause.consequence}}}
	\le 
	O\bigg(
	\f{\min\set{(\mu_e)^1(\mred),(\mu_e)^2(\mred)}}
		{2^{k\EPSDIV}}
	\bigg)\,.
	\eeq
The above estimates imply that the single-copy marginals of $\mu_{e,0}$ are close to those of $\mu_e$: more precisely, we have
	\beq\label{e:ForcedRR.initial.marginals.on.edge}
	\bigg|
	\f{(\mu_{e,0})^i(\col^i)}{(\mu_e)^i(\col^i)}-1
	\bigg|
	\le
	\begin{cases}
	k^{O(1)}/2^k &\textup{if $\col^i=\whi$,}\\
	O(2^{-k\EPSDIV}) &\textup{if $\col^i=\mred$,}\\
	\end{cases}
	\eeq
for both $i=1,2$. We then turn to estimating the marginals at $t=0$ on the edges of $\delta a\setminus e$. For $e'\in \delta a\setminus e$ and $\sigma\in\set{\yel,\cya}^2$, we can represent the marginal probability of $\sigma$ on $e'$ as
	\beq\label{e:conditional.rep.on.eprime}
	\mu_{e',0}(\sigma)
	=\sum_{\col\in\set{\whi,\red}^2}
	\mu_{e,0}(\col)
	\mu_0\Big(\sigma_{e'}=\sigma\,
		\Big|\,\col_e=\col\Big)\,.
	\eeq
The term $\mu_{e,0}(\col)$ was just estimated above,
while the conditional probability 
$\mu_0(\sigma_{e'}=\sigma\,|\,\col_e=\col)$ can be obtained from the representation \eqref{e:ForcedRR.mu.of.beta}. For instance, for each $e'\in\delta a\setminus e$, we have
	\[
	\mu_0\Big(\sigma_{e'}=\yel\yel\,
		\Big|\,\col_e=\mred\whi\Big)
	\stackrel{\eqref{e:ForcedRR.mu.of.beta}}{=}
	\f{\DS 
	\mu_{e'}(\yel\yel)
	\prod_{e''\in\delta a\setminus\set{e,e'}}
	\mu_{e''}(\yel\notr)}
	{\DS \mu_{e'}(\yel\notr) 
	\prod_{e''\in\delta a\setminus\set{e,e'}}
	\mu_{e''}(\yel\notr)}
	\Bigg\{1-O\bigg(\f1{2^{k\EPSDIV}}\bigg)\Bigg\}
	= \f{\mu_{e'}(\yel\yel)}{1/2}
		\Bigg\{1-O\bigg(\f1{2^{k\EPSDIV}}\bigg)\Bigg\}\,,
	\]
where the last estimate uses that $e'$ is a nice edge, so $\mu_{e''}(\yel\notr)$ is close to $1/2$. Similarly,
	\[\mu_0\Big(\sigma_{e'}=\yel\yel\,
		\Big|\,\col_e=\whi\whi\Big)
	\stackrel{\eqref{e:ForcedRR.mu.of.beta}}{=}
	\f{\mu_{e'}(\yel\yel)}{\mu_{e'}(\notr\notr)}
	\Bigg\{1-O\bigg(\f{k^2}{2^k}\bigg)\Bigg\}
	=\mu_{e'}(\yel\yel)
	\Bigg\{1-O\bigg(\f{k^2}{2^k}\bigg)\Bigg\}\,.
	\]
Lastly we have $\mu_0(\sigma_{e'}=\yel\yel\,|\,\col_e=\mred\mred)=1$. Substituting the last few estimates into \eqref{e:conditional.rep.on.eprime} gives
	\begin{align*}
	\mu_{e',0}(\yel\yel)
	&\stackrel{\eqref{e:conditional.rep.on.eprime}}{=}
	\mu_{e,0}(\mred\mred)
	+\mu_{e,0}\Big(\set{\mred\whi,\whi\mred}\Big)
	\f{\mu_{e'}(\yel\yel)}{1/2}
		\Bigg\{1-O\bigg(\f1{2^{k\EPSDIV}}\bigg)\Bigg\}
	+\mu_{e,0}(\whi\whi) 
	\mu_{e'}(\yel\yel)
	\Bigg\{1-O\bigg(\f{k^2}{2^k}\bigg)\Bigg\}\\
	&\stackrel{\eqref{e:ForcedRR.init.pair.mgl.on.edge}}{=}
	\mu_{e'}(\yel\yel)\Bigg\{
	O\bigg(\f1{2^{k(1+\EPSDIV)}}\bigg)
	+O\bigg( \f1{2^k}\bigg)
	+ 1-O\bigg(\f{k^2}{2^k}\bigg)\Bigg\}
	=\mu_{e'}(\yel\yel)\Bigg\{ 1-O\bigg(\f{k^2}{2^k}\bigg)\Bigg\}\,.
	\end{align*}
By similar (but simpler) calculations, the same estimate holds for the other elements $\sigma\in\set{\yel,\cya}^2$. Lastly, we estimate the marginal law at $t=0$ for the frozen spin $x_v$, using a similar conditioning as in \eqref{e:conditional.rep.on.eprime}. To this end, note that
	\[
	\mu_0\Big(x_v=x\,\Big|\,\col_e=\whi\whi\Big)
	\stackrel{\eqref{e:ForcedRR.mu.of.beta}}{=}
	\f{\DS b_{v,0}(x)
	\prod_{i=1,2}(b_{e,0})^i(\whi)}{
	\DS
	\sum_{x'\in\set{\minus,\plus,\free}^2}
	b_{v,0}(x')
	\prod_{i=1,2}(b_{e,0})^i(\whi)}
	\stackrel{\eqref{e:vv.lemma.initialization}}{=}
	\mu_v(x)
	\]
for any $x\in\set{\minus,\plus,\free}^2$. Likewise we have
	\[
	\mu_0\Big(x_v=x\,\Big|\,\col_e=\mred\whi\Big)
	\stackrel{\eqref{e:ForcedRR.mu.of.beta}}{=}
	\f{\Ind{x^1=\plus} b_{v,0}(x)
		(b_{e,0})^1(\mred)(b_{e,0})^2(\whi)}
	{\DS\sum_{x'\in\set{\plus}
		\times\set{\minus,\plus,\free}}
		b_{v,0}(x')
		(b_{e,0})^1(\mred)(b_{e,0})^2(\whi)}
	\stackrel{\eqref{e:vv.lemma.initialization}}{=}
	\f{\mu_v(x)}{(\mu_v)^1(\plus)}
	\Ind{x^1=\plus}\,,
	\] 
and a similar expression holds for the case $\col_e=\whi\red$. Lastly, note that
$\mu_0(x_v=x\,\Big|\,\col_e=\mred\mred)
=\Ind{x=\plus\plus}$. Substituting these into \eqref{e:conditional.rep.on.eprime} gives 
	\begin{align*}
	\mu_{v,0}(x)
	&\stackrel{\eqref{e:conditional.rep.on.eprime}}{=}
	\mu_{e,0}(\mred\mred)\Ind{x=\plus\plus}
	+ \mu_{e,0}(\mred\whi)\f{\mu_v(x)}
		{(\mu_v)^1(\plus)}
		{\Ind{x^1=\plus}}
	+ \mu_{e,0}(\whi\mred)\f{\mu_v(x)}
		{(\mu_v)^2(\plus)}
		{\Ind{x^2=\plus}}
	+\mu_{e,0}(\whi\whi) \mu_v(x)\\
	&\stackrel{\eqref{e:ForcedRR.init.pair.mgl.on.edge}}{=}
	\mu_v(x)
	\Bigg\{
	O\bigg(\f{\Ind{x=\plus\plus}}{2^{k(1+\EPSDIV)}}\bigg)
	+O\bigg(\f{\Ind{x^1=\plus}
		+ \Ind{x^2=\plus}}{2^k}\bigg)
	+1-O\bigg(\f{k^2}{2^k}\bigg)
	\Bigg\}
	=\mu_v(x)\Bigg\{
	1-O\bigg(\f{k^2}{2^k}\bigg)
	\Bigg\}\,.
	\end{align*}
To summarize the last few estimates, we have shown that at $t=0$ we have 
	\beq\label{e:ForcedRR.initial.marginals.off.edge}
	\begin{pmatrix}
	\DS\bigg\|\f{\mu_{v,0}}{\mu_v}-1\bigg\|_\infty
	+
	\sum_{i=1,2}
	\bigg|\f{(\mu_{e,0})^i(\whi)}
		{(\mu_e)^i(\whi)}-1\bigg|
	+
	\sum_{e'\in\delta a\setminus v}
	\bigg\|\f{\mu_{e',0}}{\mu_{e'}}-1\bigg\|_\infty
	\\
	\DS\sum_{i=1,2} \bigg|\f{(\mu_{e,0})^i(\mred)}
		{(\mu_e)^i(\mred)}-1\bigg|
		\end{pmatrix}
	\le k^{O(1)}
	 \begin{pmatrix}
		2^{-k}\\ 2^{-k\EPSDIV} \end{pmatrix}\,.
	\eeq
We now turn to the construction of the sequence $b_t\to b_\infty\equiv b$, where we recall that $b$ denotes the limiting weights in \eqref{e:ForcedRR.mu.of.beta} that give the solution to the constrained entropy maximization problem.\smallskip

\noindent\bemph{Step 3. Iterative construction and covariance estimates.} For each $t\ge0$, suppose we have the weights $b_t$, which define the measure $\mu_t$. We then define the weights at time $t+1$ by setting
	\beq\label{e:update.rule.on.central.edge}
	\f{b_{v,t+1}(x)}{b_{v,t}(x)}
		=\f{\mu_v(x)}{\mu_{v,t}(x)}\,,\quad
	\f{b_{e',t+1}(\sigma)}{b_{e',t}(\sigma)}
		=\f{\mu_{e'}(\sigma)}{\mu_{e',t}(\sigma)}
			\,,\quad
	\f{(b_{e,t+1})^i(\col^i)}{(b_{e,t})^i(\col^i)}
		=\f{\mu_e(\col^i)}{\mu_{e,t}(\col^i)}
	\eeq
for all $x\in\set{\minus,\plus,\free}^2$,
 $e'\in\delta a\setminus e$,
 $\sigma\in\set{\yel,\cya}^2$, 
 $i=1,2$, and $\col^i\in\set{\whi,\mred}$. We will show by induction that for all $t\ge0$ we have the bound
	\beq\label{e:tree.a.v.induct}
	\begin{pmatrix} \DS
	\bigg\|
		\f{\mu_{v,t}}{\mu_v}
		-1\bigg\|_\infty
	+\sum_{i=1,2}
	\bigg|\f{(\mu_{e,t})^i(\whi)}
		{(\mu_e)^i(\whi)}-1\bigg|
	+\sum_{e'\in\delta a\setminus e}
	\bigg\| \f{\mu_{e',t}}{\mu_{e'}}-1\bigg\|_\infty\\
	\DS \sum_{i=1,2}
		\bigg|\f{(\mu_{e,t})^i(\mred)}
		{(\mu_e)^i(\mred)}-1\bigg|
	\end{pmatrix}
	\le \f{k^{O(1)}}{(2^{k\EPSDIV})^{t/2}}
	\begin{pmatrix}2^{-k}\\2^{-k\EPSDIV}
		\end{pmatrix}\,,
	\eeq
where the base case $t=0$ follows from \eqref{e:ForcedRR.initial.marginals.off.edge}.
Note that
\eqref{e:update.rule.on.central.edge} and \eqref{e:tree.a.v.induct}
together imply
	\beq\label{e:induct.consequence.on.sum.b}
	\begin{pmatrix}
	\DS \Bigg\{
		\bigg\|\f{b_{v,t+1}}{b_{v,t}}-1\bigg\|_\infty
		+\sum_{i=1,2}
		\bigg|\f{(b_{e',t+1})^i(\whi)}
			{(b_{e',t})^i(\whi)}
			-1\bigg|
		+\sum_{e'\in\delta a\setminus e}
		\bigg\|\f{b_{e',t+1}}{b_{e',t}}-1\bigg\|_\infty
		\Bigg\}\\
	\DS\sum_{i=1,2}
		\bigg|\f{(b_{e',t+1})^i(\mred)}
			{(b_{e',t})^i(\mred)}
			-1\bigg|
	\end{pmatrix}
	\le \f{k^{O(1)}}{(2^{k\EPSDIV})^{t/2}}
	\begin{pmatrix}2^{-k}\\2^{-k\EPSDIV}
		\end{pmatrix}\,,
	\eeq
Towards the proof of \eqref{e:tree.a.v.induct},
we estimate covariances under the measure $\mu_t$. First, for all $x\in\set{\minus,\plus,\free}^2$ we have 	\begin{align*}
	\f{Z_t \mu_t(x_v=x,(\col_e)^1=\whi)}
		{\mu_v(x) (\mu_e)^1(\whi)}
	&=\f{b_{v,t}(x)
	(b_{e,t})^1(\whi)}{\mu_v(x) (\mu_e)^1(\whi)}
	\Bigg\{
	(b_{e,t})^2(\whi)
	\prod_{e'\in\delta a\setminus e}
		b_{e',t}(\notr\notr)
		\bigg(1-\f{k^{O(1)}}{2^k}\bigg)\\
	&\qquad\qquad\qquad
		+\Ind{x^2=\plus}(b_{e,t})^2(\mred)
		\prod_{e'}b_{e',t}(\notr\yel)
		\bigg(1-\f{O(1)}{2^{k\EPSDIV}}\bigg)
	\Bigg\}\\
	&=\f{b_{v,t}(x)
	(b_{e,t})^1(\whi)}{\mu_v(x) (\mu_e)^1(\whi)}
	\Bigg\{1-\f{k^{O(1)}}{2^k}\Bigg\}
	\stackrel{\eqref{e:induct.consequence.on.sum.b}}{=}
	1-\f{k^{O(1)}}{2^k}\,.
	\end{align*}
In combination with \eqref{e:tree.a.v.induct}, this implies that for all $x\in\set{\minus,\plus,\free}^2$ and $i=1,2$ we have
	\[\Bigg|
	\f{\Cov_{\mu_t}(
		\Ind{x_v=x},
		\Ind{(\col_e)^i=\whi})}
		{\mu_v(x) (\mu_e)^i(\whi)}\Bigg|
	=\Bigg| \f{\mu_t(x_v=x,(\col_e)^1=\whi)}
		{\mu_v(x) (\mu_e)^1(\whi)}
		-\f{\mu_{v,t}(x) (\mu_{e,t})^1(\whi)}
		{\mu_v(x) (\mu_e)^1(\whi)}
		\Bigg|
	\le \f{k^{O(1)}}{2^k}\,.
	\]
Next, for $e'\in\delta a\setminus e$ and
$\sigma\in\set{\yel,\cya}^2$, we calculate
	\begin{align*}
	\f{Z_t \mu_t((\col_e)^1=\whi,
		\sigma_{e'}=\sigma)}
		{(\mu_e)^1(\whi)\mu_{e'}(\sigma)}
	&=
	\f{(b_{e,t})^1(\whi)b_{e',t}(\sigma)}
		{(\mu_e)^1(\whi)\mu_{e'}(\sigma)}
	\Bigg\{
	(b_{e,t})^2(\whi)
	\prod_{e''\in\delta a\setminus\set{e,e'}}
	b_{e'',t}(\notr\notr)
	\bigg(1-\f{k^{O(1)}}{2^k}\bigg) \\
	&\qquad\qquad\qquad
	+\Ind{\sigma^2=\yel}
	(b_{e,t})^2(\mred)
	(b_{v,t})^2(\plus)
	\prod_{e''\in\delta a\setminus\set{e,e'}}
	b_{e'',t}(\notr\yel)
	\bigg(1-\f{O(1)}{2^{k\EPSDIV}}\bigg)
	\Bigg\}\\
	&=	\f{(b_{e,t})^1(\whi)b_{e',t}(\sigma)}
		{(\mu_e)^1(\whi)\mu_{e'}(\sigma)}
	\Bigg\{1- \f{k^{O(1)}}{2^k}\Bigg\}
	\stackrel{\eqref{e:induct.consequence.on.sum.b}}{=}
	1- \f{k^{O(1)}}{2^k}\,.
	\end{align*}
This implies that for all $e'\in\delta a\setminus e$ and $\sigma\in\set{\yel,\cya}^2$, we have
	\[
	\f{\Cov_{\mu_t}(
		\Ind{(\col_e)^i=\whi},
		\Ind{\sigma_{e'}=\sigma})}
		{(\mu_e)^i(\whi)
		\mu_{e'}(\sigma)
		 }
		\le\f{k^{O(1)}}{2^k}\,.
	\]
Similar calculations for the case $(\col_e)^i=\mred$ give 
	\begin{align*}
	\f{\Cov_{\mu_t}(\Ind{x_v=x},
		\Ind{(\col_e)^i=\mred})}
		{\mu_v(x)
		(\mu_e)^i(\mred)}
	&\le O(1)\,,\\
	\f{\Cov_{\mu_t}(
		\Ind{(\col_e)^i=\mred},
		\Ind{\sigma_{e'}=\sigma})}
			{(\mu_e)^i(\mred)
		\mu_{e'}(\sigma)}
	&\le O(1)\,,
	\end{align*}
for all $x\in\set{\minus,\plus,\free}^2$,
 $e'\in\delta a\setminus e$, and $\sigma\in\set{\yel,\cya}^2$. Next, on the edge $e=(av)$, we have
 	\[
	\f{Z_t\mu_{e,t}(\whi\whi)}
		{(\mu_e)^1(\whi)(\mu_e)^2(\whi)}
	= \f{(b_{e,t})^1(\whi)(b_{e,t})^2(\whi)}
		{(\mu_e)^1(\whi)(\mu_e)^2(\whi)}
	\prod_{e'\in\delta a\setminus e}
		b_{e',t}(\notr\notr)
	\Bigg\{1-\f{k^{O(1)}}{2^k}\Bigg\}
	\stackrel{\eqref{e:induct.consequence.on.sum.b}}{=}
	1- \f{k^{O(1)}}{2^k}\,.
	\]
Next, using the diversity bound
\eqref{e:diverse.clause.consequence}, we also have
	\begin{align*}
	\f{Z_t\mu_{e,t}(\mred\whi)}
		{(\mu_e)^1(\mred)(\mu_e)^2(\whi)}
	&\stackrel{\eqref{e:diverse.clause.consequence}}{=} 
	\f{(b_{e,t})^1(\mred)(b_{e,t})^2(\whi)}
		{(\mu_e)^1(\mred)(\mu_e)^2(\whi)}
	(b_{v,t})^1(\plus) 
	\prod_{e'\in\delta a\setminus e}
		b_{e',t}(\yel\notr)
	\Bigg\{1-\f{O(1)}{2^{k\EPSDIV}}\Bigg\} \\
	&
	\stackrel{\eqref{e:induct.consequence.on.sum.b}}{=}
	\Bigg\{1-\f{O(1)}{2^{k\EPSDIV}}\Bigg\}
	2^K (b_{v,t})^1(\plus) 
	\prod_{e'\in\delta a\setminus e}
		b_{e',t}(\yel\notr)
	\stackrel{\eqref{e:induct.consequence.on.sum.b}}{=}
	1-\f{O(1)}{2^{k\EPSDIV}}\,,\\
	\f{Z_t\mu_{e,t}(\mred\mred)}
		{(\mu_e)^1(\mred)(\mu_e)^2(\mred)}
	&= \f{(b_{e,t})^1(\mred)(b_{e,t})^2(\mred)}
		{(\mu_e)^1(\mred)(\mu_e)^2(\mred)}
	\mu_v(\plus\plus)
	\prod_{e'\in\delta v\setminus e}
	b_{e',t}(\yel\yel)\\
	&\stackrel{\eqref{e:induct.consequence.on.sum.b}}{=} 2^K \cdot 2^K\cdot
	\mu_v(\plus\plus)
	\prod_{e'\in\delta v\setminus e}
	b_{e',t}(\yel\yel)
	\stackrel{\eqref{e:diverse.clause.consequence}}{\le }
	O\bigg( \f{2^k}{2^{k\EPSDIV}}\bigg)
	\,.\end{align*}
The last few estimates combined imply
	\[
	\f{\Cov_{\mu_t}(\Ind{(\col_e)^1=\col^1},
	\Ind{(\col_e)^2=\col^2})}
		{(\mu_e)^1(\col^1)
		(\mu_e)^2(\col^2)}
	\le k^{O(1)} \begin{cases}
	2^{-k}&\textup{if $\col=\whi\whi$,}\\
	2^{-k\EPSDIV}&\textup{if $\col\in\set{\mred\whi,
		\whi\mred}$,}\\
	2^{k(1-\EPSDIV)}
	&\textup{if $\col=\mred\mred$.}
	\end{cases}
	\]
Lastly, it is straightforward to verify (details omitted) that
	\[\f{\Cov_{\mu_t}(
		\Ind{x_v=x},
		\Ind{\sigma_{e'}=\sigma})}
			{ \mu_v(x) \mu_{e'}(\sigma)}
	\le \f{k^{O(1)}}{2^k}
	\]
for all 
 $x\in\set{\minus,\plus,\free}^2$,
$e'\in\delta a \setminus e$, and $\sigma\in\set{\yel,\cya}^2$.\smallskip

\noindent\bemph{Step 4. Estimates on Lagrangian weights.}
Now recall from 
\eqref{e:derivative.is.covariance}
that derivatives of $\mu_t$-marginals with respect to $b_t$ can be expressed as covariance. Let $\tilde{\mu}_{t+1}$ be defined by
	\[
	\tilde{\mu}_{t+1}
	\Big(x_v,\col_e,\usi_{\delta a\setminus e}\Big)	
	=
	\f1{\tilde{Z}_{t+1}}
	b_{v,t}(x)
	(b_{e,t+1})^1((\col_e)^1)
	(b_{e,t})^2((\col_e)^2)
	\prod_{e'\in\delta a\setminus e}
		b_{e',t}(\sigma_{e'})\,,
	\]
where $\tilde{Z}_{t+1}$ is the normalizing constant. It follows from the update rule \eqref{e:update.rule.on.central.edge} that
the marginal of $\tilde{\mu}_{t+1}$ on $(\col_e)^1$
is exactly the desired marginal $(\mu_e)^1$. Consequently,
	\begin{align*}
	\delta_{t+1}(\col^1) &\equiv \bigg|
	\f{(\mu_{e,t+1})^1(\col^1)}{(\mu_e)^1(\col^1)}-1
	\bigg|
	=\bigg|
	\f{(\mu_{e,t+1})^1(\col^1)}
		{(\tilde{\mu}_{e,t+1})^1(\col^1)}-1\bigg|
	\\
	&\le
	\f1{(\mu_e)^1(\col^1)}
	O\Bigg(
	\sum_x
	\f{\pd\mu_{e,t+1}(\col^1)}
		{\pd\log b_{v,t+1}(x)}
	\bigg| \f{b_{v,t+1}(x)}{b_{v,t}(x)}-1\bigg|
	+\sum_{\col^2}
	\f{\pd\mu_{e,t+1}(\col^1)}
		{\pd\log (b_{e,t+1})^2(\col^2)}
	\bigg| \f{(b_{e,t+1})^2(\col^2)}
		{(b_{e,t})^2(\col^2)}
	-1\bigg|
	\\
	&\qquad\qquad\qquad\qquad\qquad\qquad\qquad
	+\sum_{e',\sigma}
	\f{\pd\mu_{e,t+1}(\col^1)}
		{\pd\log b_{e',t+1}(\sigma)}
		\bigg|
		\f{b_{e',t+1}(\sigma)}
			{b_{e',t}(\sigma)}-1\bigg|
	\Bigg)\,.
	\end{align*}
Substituting the preceding covariance estimates into the last bound and applying \eqref{e:induct.consequence.on.sum.b} gives
	\begin{align*}
	\delta_{t+1}(\whi)
	&\le
	k^{O(1)}\Bigg\{
	\f1{2^k}
	\cdot
	\f{1}{2^k(2^{k\EPSDIV/2})^t}
	+	\f1{2^{k\EPSDIV}} \cdot
	(\mu_e)^2(\mred)
	\cdot
	\bigg|\f{(b_{e,t+1})^2(\mred)}
		{(b_{e,t})^2(\mred)}-1 \bigg|
	\Bigg\}\\
	&\le k^{O(1)}\Bigg\{
	\f1{2^k}
	\cdot
	\f{1}{2^k(2^{k\EPSDIV/2})^t}
	+	\f1{2^{k\EPSDIV}}
	\cdot \f1{2^k}
	\cdot \f1{2^{k\EPSDIV}((2^{k\EPSDIV/2})^t}
	\Bigg\}
	\le \f{k^{O(1)}}{2^k(2^{k\EPSDIV/2})^{t+2}}\,,\\
	\delta_{t+1}(\mred)
	&\le
	k^{O(1)}\Bigg\{
	\f1{2^{k\EPSDIV}((2^{k\EPSDIV/2})^t}
	+ 2^{k(1-\EPSDIV)}
	\cdot(\mu_e)^2(\mred)
	\cdot \f1{2^{k\EPSDIV}((2^{k\EPSDIV/2})^t}
	\Bigg\}
	\le 
	\f{k^{O(1)}}{2^{k\EPSDIV}(2^{k\EPSDIV/2})^{t+2}}\,.
	\end{align*}
This verifies the inductive hypothesis \eqref{e:tree.a.v.induct} for the quantities
	\[
	\delta_{t+1}(\col^i)
	= \bigg|\f{(\mu_{e,t+1})^i(\col^i)}
		{(\mu_e)^i(\col^i)}-1\bigg|\,.
	\]
The remaining estimates in \eqref{e:tree.a.v.induct} follow by similar calculations (details omitted). It follows that as $t\to\infty$ the weights $b_t$ converge to the desired limiting weights $b_\infty\equiv b$ that define the optimal measure $\mu$ in \eqref{e:ForcedRR.mu.of.beta}. It then follows from the bound \eqref{e:induct.consequence.on.sum.b} that
	\begin{align*}
	\mu_e(\mred\mred)
	&\stackrel{\eqref{e:ForcedRR.mu.of.beta}}{=}
	O\Bigg(
	b_v(\plus\plus)
	\Bigg\{\prod_{i=1,2} (b_e)^i(\mred)\Bigg\}
	\Bigg\{ \prod_{e'\in\delta a\setminus e}
		b_{e'}(\yel\yel)\Bigg\}
		\Bigg)\\
	&\stackrel{\eqref{e:induct.consequence.on.sum.b}}{=}
	O\Bigg(
	\mu_v(\plus\plus)
	\prod_{i=1,2}\Big(
	(\mu_e)^i(\mred) 2^K\Big)
	\prod_{e'\in\delta a\setminus e}
		\mu_{e'}(\yel\yel)\Bigg)
	\stackrel{\eqref{e:diverse.clause.consequence}}{\le}
	\f1{2^{k(1+\EPSDIV)}}\,,
	\end{align*}
as claimed.
\end{proof}
\end{lem}

\begin{lem}\label{l:large11Case}
Suppose $v$ is a non-defective variable of type $\bT$, satisfying the bounds $\pi_v(\plus\plus)
	=\pi_{\bT}(\plus\plus)\ge 2^{-k/16}$ and
	\beq\label{e:notDiverse.small}
	\notDiverse(v)
	\stackrel{\eqref{e:exp.num.nondiv.nbr.clauses}}{=}
	\sum_{\bt\in\bT}\sum_{\bL}
	\Ind{\bL\notin\mathbb{D}}
	\pi_{\DD}(\bL\,|\,\bt)
	\le 2^k\,.
	\eeq
Then, for every edge $e\in\delta v(\plus)$, we have (with $\EPSDIV$ as in Lemma~\ref{l:ForcedRR}) the bound
	\[
	\max_{\bL,j}
	\bigg\{\omega_{\bL,j}(\red\red)
	: 
	\textup{$\bL\in\mathbb{D}$
	with $\bL(j)=\bt_e$}
	\bigg\}
	\le \f1{2^{k(1+\EPSDIV)}}\,.
	\]
The same statement holds if we replace $\plus$ with $\minus$ throughout.

\begin{proof} We claim that
	\beq\label{enotEnoughReds}
	\min\Bigg\{
	\sum_{ e \in\delta v(\plus)} 
		\pi_e\Big( \set{\red\red,\red\blu}\Big),
	\sum_{ e \in\delta v(\plus)} 
		\pi_e\Big( \set{\red\red,\blu\red}\Big)
	\Bigg\}
	\ge \f{k \pi_v(\plus\plus)}{300}\,.
	\eeq
Let us first note that
\eqref{enotEnoughReds} implies the result of the lemma: indeed, if \eqref{enotEnoughReds} holds,
then condition
\eqref{e:assumption.rr.avg} of Lemma~\ref{l:rrUnforced} is satisfied with $\EPSTWO=1/300$, and applying that lemma gives
	\[
	\omega_{\bL,j}(\red\red)
	\le \f{\bom_{\bL,j}(\ups=\mred\mred)}
		{1- 2^{-k\EPSTHREE}}
	\le O\Big( \bom_{\bL,j}
		(\ups=\mred\mred)\Big)
	\]
for $j=j(\bt)$ and all $\bL$ such that $\bL(j)=\bt$.
If in addition $\bL\in\mathbb{D}$, then
combining with Lemma~\ref{l:ForcedRR} gives
	\[
	\omega_{\bL,j}(\red\red)
	\le O\Big( \bom_{\bL,j}
		(\ups=\mred\mred)\Big)
	\le O\bigg( \f1{2^{k(1+\EPSDIV)}}\bigg)\,.
	\]
Thus it suffices to prove \eqref{enotEnoughReds}.\smallskip

\noindent\bemph{Step 1. Preliminary bounds.}
 For $e\in\delta v(\plus)$ we have $\pi_e(\blu\blu)= \pi_v(\plus\plus)-O(2^{-k})\ge 2^{-k/15}$, so Lemma~\ref{l:indifferent} gives
	\[
	\omega_{\bL,j}(\blu\blu)
	\ge \f{\pi_e(\blu\blu)}8
	\ge \f18
	\Bigg\{ \pi_v(\plus\plus)
		-O\bigg(\f1{2^k}\bigg)
		\Bigg\}
	\]
for all $\bL,j$ with $\bL(j)=\bt_e$. Combining with
Lemma~\ref{l:lotsOfAs} gives
	\beq\label{e:large11Case.many.bb.ww}
	\bom_{\bL,j}
		\Bigg(
		\hspace{-3pt}\begin{array}{c}\sigma=\blu\blu,\\
		\col=\whi\whi\end{array}\hspace{-3pt}
		\Bigg)
	\ge \omega_{\bL,j}(\blu\blu)
		-O\bigg(\f{k^2}{2^k}\bigg)
	\ge 
	\omega_{\bL,j}(\blu\blu)
	\Big\{ 1-o_k(1)\Big\}
	\,.\eeq
Suppose for contradiction that \eqref{enotEnoughReds} fails. Since $v$ is non-defective, it must also be nice (Definition~\ref{d:nice}), so
	\[
	\sum_{e\in\delta v(\plus)}
	\pi_e\Big(\set{\red\yel,\red\grn}\Big)
	\ge \sum_{e\in\delta v(\plus)}
	(\pi_e)^1(\red) - 
		\f{k \pi_v(\plus\plus)}{300}
	\ge \f{k 2^{k-1} \log 2}{2^k}
		\Big\{1-o_k(1)\Big\}
	- \f{ k \pi_v(\plus\plus)}{300}
	\ge \f{k}{3}\,.
	\]
Therefore, it holds for some
$\tau\in\set{\yel,\grn}$ that 
	\beq\label{e:too.many.ry}
	\sum_{e\in\delta v(\plus)}
	\pi_e(\red\tau) \ge \f{k}{6}
	\eeq
--- suppose this is the case for $\tau=\yel$. 
Then Lemma~\ref{l:ryUnforced} 
(whose condition 
\eqref{e:ryUnforced.condition} 
is satisfied, due to
\eqref{e:too.many.ry}) gives
	\[\bom_{\bL,j}
	\Bigg(\hspace{-3pt}
		\begin{array}{c}\sigma=\red\yel,\\
		\ups^1=\mred\end{array}\hspace{-3pt}
		\Bigg) 
	\ge \omega_{\bL,j}(\red\yel)
	\bigg\{ 1-\f1{2^{k\EPSTHREE}}\bigg\}
	\]
for all $e\in\delta v(\plus)$
and all $\bL(j)=\bt_e$. If $\bL$ is also diverse, then combining with Lemma~\ref{l:lotsOfAsDiverse} gives 
	\beq\label{e:lbd.movable.red.white}
	\bom_{\bL,j}
		\Bigg(
		\hspace{-3pt}
		\begin{array}{c}\sigma=\red\yel,\\
		\col=\mred\whi
		\end{array}\hspace{-3pt}
		\Bigg) 
	\ge \bom_{\bL,j}\Bigg(\hspace{-3pt}
		\begin{array}{c}\sigma=\red\yel,\\
		\ups^1=\mred\end{array}\hspace{-3pt}
		\Bigg) 
	-\bom_{\bL,j}\Bigg(
		\hspace{-3pt}
		\begin{array}{c}
		\sigma=\red\yel,\\
		\varsigma^1\ne\whi,
		\varsigma^2\ne\whi
		\end{array}\hspace{-3pt}
		\Bigg)
	\ge \bom_{\bL,j}(\red\yel)
		\bigg\{ 1-\f1{2^{k\EPSTHREE}}\bigg\}
		-\f1{2^{k(1+\EPSONE)}}\,.\eeq
Note that \eqref{e:lbd.movable.red.white} used the assumption that $v$ is non-defective, which implies that $\bL$ is nice (i.e., neighbors only nice variables) whenever $\bL(j)=\bt_e$ for $e\in\delta v$. We now define a probability measure $\bm{\pi}$ on the space of pairs $(e,\bL)$, where $e\in\delta v(\plus)$ and $\bt_e\in\bL$, such that
	\beq\label{e:defn.bm.pi}
	\bm{\pi}(e,\bL)
	\equiv
	\f{\Ind{e\in\delta v(\plus)}}
		{|\delta v(\plus)|}
		\pi_{\DD}(\bL\,|\,\bt_e)\,.
	\eeq
We can equivalently regard $\bm{\pi}$ as a probability measure on pairs $(\bt,\bL)$ where $\bt=\bt_e$. Now, on this probability space, define the random variable
	\[X\equiv X(e,\bL)
	\equiv 
	\Ind{\bL\in\mathbb{D}}
	\f{\omega_{\bL,j}(\red\yel)}
	{\starpi_e(\red)}
	\equiv X(\bt,\bL)\,,\]
and note that $0\le X\le 1$ with probability one.
The expectation of $X$ with respect to $\bm{\pi}$ is
	\begin{align*}
	\mathbf{E}_{\bm{\pi}} X
	&\stackrel{\eqref{e:defn.bm.pi}}{=}
		\f1{|\delta v(\plus)|}
	\Bigg\{ \sum_{e\in\delta v(\plus)}
	\sum_{\bL}
	\pi_{\DD}(\bL\,|\,\bt_e)
	\f{\omega_{\bL,j}(\red\yel)}
	{\starpi_e(\red)}
	-O\Bigg(
	\sum_{e\in\delta v(\plus)}
	\sum_{\bL}\Ind{\bL\notin\mathbb{D}}
	\pi_{\DD}(\bL\,|\,\bt_e)\Bigg)
	\Bigg\}\\
	&\stackrel{\eqref{e:notDiverse.small}}{\ge}
	\f1{|\delta v(\plus)|} \sum_{e\in\delta v(\plus)}
	\sum_{\bL}
	\pi_{\DD}(\bL\,|\,\bt_e)
	\f{\omega_{\bL,j}(\red\yel)}{\starpi_e(\red)}
	- \f1k
	=\f1{|\delta v(\plus)|} \sum_{e\in\delta v(\plus)}
	\f{\pi_e(\red\yel)}{\starpi_e(\red)}
	- \f1k\\
	&\stackrel{\eqref{e:too.many.ry}}{\ge} 
	\Big\{ 1-o_k(1)\Big\}
	\f{2^{k-1} \cdot k/6}{k 2^{k-1}\log2}
	\ge \f15\,.
	\end{align*}
On the other hand, since we noted above that $0\le X\le 1$, we have
	\[
	\mathbf{E}_{\bm{\pi}} X
	\le \bm{\pi}\bigg(X \ge \f1{10}\bigg)
	+ \f1{10}
	\Bigg\{1-\bm{\pi}\bigg(X \ge \f1{10}\bigg)\Bigg\}\,,
	\]
and rearranging gives
$\bm{\pi}(X \ge 1/10)\ge 1/9$. On the event $X(e,\bL)\ge1/10$, it follows from \eqref{e:lbd.movable.red.white} that
	\beq\label{e:large11Case.many.ry.moveable}
	\bom_{\bL,j}
	\Bigg(\hspace{-3pt}\begin{array}{c}
		\sigma=\red\yel,\\
		\col=\mred\whi\end{array}\hspace{-3pt}
		\Bigg)
	\ge \omega_{\bL,j}(\red\yel)
	\Big\{1-o_k(1)\Big\}\,.\eeq
In the next step we will perform a switching argument between \eqref{e:large11Case.many.bb.ww}
and \eqref{e:large11Case.many.ry.moveable}.\smallskip

\noindent\bemph{Step 2. Edge switching argument.} Let $\GG=(V,F,E)$ be any (processed) $\ksat$ graph. Fix a clause type $\bL$ and index $1\le j\le k(\bL)$, and let $E(\bL,j)$ denote the subset of all edges
$e=(av)\in E$ such that $\bL_a=\bL$ and $j(\bt_e)=j$. Suppose $\usi$ is a valid pair coloring on $\GG$, and let $\ucol$ be the corresponding configuration from Definition~\ref{d:white.violet}. Suppose we have two edges $e=(av)$ and $e'=(a'v')$ in $E(\bL,j)$, such that
	\[
	\mathbf{X}_{av} \equiv
	\begin{pmatrix} \sigma_{av}\\
		\col_{av}\end{pmatrix}
	= \begin{pmatrix} \blu\blu \\
		\whi\whi\end{pmatrix}
	\equiv A_1\,,\quad
	\mathbf{X}_{a'v'}
	\equiv\begin{pmatrix} \sigma_{a'v'}\\
		\col_{a'v'}\end{pmatrix}
	= \begin{pmatrix} \red\yel \\
		\mred\whi\end{pmatrix}
	\equiv A_2\,.
	\]
If we cut the edges $e,e'$ and form new edges 
$(a'v)$, $(av')$, then a valid configuration
$(\usi',\ucol')$ on the switched graph $\GG'$ is given by setting
	\[\mathbf{X}_{a'v}
	\equiv\begin{pmatrix} \sigma_{a'v}\\
		\col_{a'v}\end{pmatrix}
	= \begin{pmatrix}
		\red\blu\\ \mred\whi\end{pmatrix}
	\equiv B_1\,,\quad
	\mathbf{X}_{av'}
	\equiv\begin{pmatrix} \sigma_{av'}\\
		\col_{av'}\end{pmatrix}
	=\begin{pmatrix}\blu\yel \\
		\whi\whi\end{pmatrix}
	\equiv B_2\,,
	\]
keeping all other colors unchanged. Moreover, the switching preserves all single-copy marginals, 
so $\usi'$ is a judicious configuration on $\GG'$.
If $X(\bt,\bL)\ge1/10$ for $\bt=\bL(j)$, then
	\begin{align*}
	p(A_1)
	&\equiv
	\f{|\set{e \in E(\bL,j) : \mathbf{X}_{av}=A_1}|}
		{|E(\bL,j)|}
	\stackrel{\eqref{e:large11Case.many.bb.ww}}
	{\ge} \bigg\{1-o_k(1)\bigg\}
	\omega_{\bL,j}(\blu\blu)\,,\\
	p(A_2)
	&\equiv
	\f{|\set{e \in E(\bL,j) : \mathbf{X}_{av}=A_2}|}
		{|E(\bL,j)|}
	\stackrel{\eqref{e:large11Case.many.ry.moveable}}{\ge}\bigg\{1-o_k(1)\bigg\}
	\omega_{\bL,j}(\red\yel)\,.\end{align*}
Suppose $\bom$ (the empirical measure on the augmented spins $(\sigma,\col)$) gives the maximal 
second moment contribution $\E_{\DD}\ZZ^2(\bom)$, 
subject to the restriction that its projection $\omega$ lies in $I_0$. Let $\bm{\Omega}$ denote the space of all (valid) pairs $(\GG,\usi)$ that are 
consistent with this $\bom$, and consider the probability measure
	\[
	\bP(\GG,\usi)
	= \f{\P_{\DD}(\GG) \Ind{(\GG,\usi)\in\bm{\Omega}}}
		{\E_{\DD} \ZZ^2(\bom)}\,.
	\]
Let $(\GG,\usi)\in\bm{\Omega}$ be sampled according to measure $\bP$, and define the subsets of edges
	\begin{align*}
	E_A&\equiv E_A(\GG,\usi)
	\equiv
	\bigg\{e\in E(\bL,j): \bm{X}_e \in\set{A_1,A_2}\bigg\}\,,\\
	E_B&\equiv E_B(\GG,\usi)
	\equiv
	\bigg\{e\in E(\bL,j): \bm{X}_e \in\set{B_1,B_2}\bigg\}\,.
	\end{align*}
If we rematch the edges within $E_A$, and
also rematch the edges within $E_B$ (uniformly at random), the resulting $(\GG',\usi')$ will also be distributed roughly according to $\bP$. The number of $B_1$-edges in $(\GG',\usi')$
is lower bounded by the number of switched edges from $E_A(\GG,\usi)$, so we conclude
	\[
	\f{p(B_1)}{1-o_n(1)}
	\ge
	\f{|E_A(\GG,\usi)|}{|E(\bL,j)|}
	\f{p(A_1) p(A_2)}{p(A_1)+p(A_2)}
	\ge
	\Big\{1-o_n(1)\Big\}
	p(A_1) p(A_2)\,.
	\]
Recalling the definition of the $A_i,B_i$, the above can be rewritten as
	\[\omega_{\bL,j}(\red\blu)
	\ge \f{\omega_{\bL,j}(\blu\blu)
		\omega_{\bL,j}(\red\yel)}{1+o_k(1)}
	\ge \f{\pi_v(\plus\plus)
	\starpi_e(\red)}{11}\,,
	\]
where the last step uses the assumption that $X(\bt,\bL)\ge 1/10$. It follows that
	\begin{align*}
	\sum_{e\in\delta v(\plus)}\pi_e(\red\blu)
	&\ge\sum_{e\in\delta v(\plus)}
	\sum_{\bL}
	\mathbf{1}\bigg\{X(e,\bL)\ge \f1{10}\bigg\}
	\pi_{\DD}(\bL\,|\,\bt_e) 
	\omega_{\bL,j}(\red\blu)\\
	&= |\delta v(\plus)|
	\bm{\pi}\bigg(X \ge \f1{10}\bigg) 
	\f{\pi_v(\plus\plus)
	\starpi_e(\red)}{11}
	\ge \f{k2^{k-1}\log2}{9[1-o_k(1)]}
	\f{\pi_v(\plus\plus)}{11 \cdot 2^k}
	\ge \f{k \pi_v(\plus\plus)}{300}\,,
	\end{align*}
proving \eqref{enotEnoughReds}. A very similar argument proves \eqref{enotEnoughReds} in the case that \eqref{e:too.many.ry} holds for $\tau=\grn$ instead of $\tau=\yel$. As explained above, 
\eqref{enotEnoughReds} implies the lemma, so this concludes the proof.
\end{proof} 
\end{lem}

\begin{cor}[used only in proof of Proposition~\ref{p:noNonDiverse}]
\label{c:RRFreq}
Consider the setting of Lemma~\ref{l:large11Case},
but without assuming a lower bound on $\pi_v(xx)$.
Then
	\[
	\sum_{e\in\delta v} 
	\sum_{\bL\in\mathbb{D}}
	\pi_{\DD}(\bL\,|\,\bt_e)
	\omega_{\bL,j(\bt_e)} 
	(\red\red)
	\le \f{k^2}{2^{k\EPSDIV}}
	\]
for $\EPSDIV$ as in Lemmas~\ref{l:ForcedRR}~and~\ref{l:large11Case}.

\begin{proof}
In view of Lemma~\ref{l:large11Case} it suffices to consider the case $\pi_v(\plus\plus) \le 2^{-k/16}$.
Suppose the conclusion of this corollary fails; more explicitly, suppose that
	\[
	\sum_{e\in\delta v} 
	\sum_{\bL\in\mathbb{D}}
	\pi_{\DD}(\bL\,|\,\bt_e)
	\omega_{\bL,j(\bt_e)} 
	(\red\red)
	\ge \f{k^4}{2^{k/16}}\,.
	\]
This implies that the assumption
\eqref{e:assumption.rr.avg} of Lemma~\ref{l:rrUnforced} holds, since
	\[
	\sum_{e\in\delta v} 
	\pi_e(\red\red\,|\, 
	\sigma\in\set{\red,\blu}^2)
	\ge \f{k^4/2^{k/16}}{\pi_v(\plus\plus)}
	\ge k^4\,.
	\]
In this case, combining 
Lemma~\ref{l:rrUnforced} and Lemma~\ref{l:ForcedRR} gives that for all $\bL,j$ with $\bL(j)=\bt_e$ for $e\in\delta v(\plus)$,
	\[
	\omega_{\bL,j}(\red\red)
	\le
	\f{\omega_{\bL,j}(\mred\mred)}{1-o_k(1)}
	\le
	\f1{2^{k(1+\EPSDIV)}}\,,
	\]
from which the conclusion of this corollary follows.
\end{proof}
\end{cor}

\subsection{Non-defective variables neighboring only diverse light clauses}
\label{ss:proof.propn.balancedeVertex}

The main result of this subsection is the following proposition:

\begin{ppn}[used only in proof of Proposition~\ref{p:apriori}]
\label{p:balancedeVertex}
In the notation of Lemma~\ref{l:large11Case},
suppose $v$ is a non-defective variable with $\notDiverse(v)=\notLight(v)=0$. Then for every $e\in\delta v$, and every
$\bL\ni_j\bt_e$ we have 
	\beq\label{e:balancedeVertex}
	\bigg|\f{\omega_{\bL,j}(\sigma)}
	{\starpi_{\bL(j)}(\sigma^1)\starpi_{\bL(j)}(\sigma^2)}
	-1\bigg| \le \f{k^{O(1)}}{2^{k\EPSLIGHT}}\,,
	\eeq
for all $\sigma\in\set{\RYGB}^2\setminus\set{\red\red}$ and all $\bL$,
with $\EPSLIGHT$ as in Lemma~\ref{l:lotsOfAsDiverse}.
\end{ppn}

\noindent The \hyperlink{pf:p.balancedeVertex}{proof of Proposition~\ref{p:balancedeVertex}} appears at the end of this subsection. Its main ingredient is the next lemma:

\begin{lem}[used only in proof of Proposition~\ref{p:balancedeVertex}]
\label{l:ZETA.conditions.balanced.vertex}
Let $\bT$ be a non-defective variable type. 
Consider the second moment of judicious configurations under $\P_{\DD}$, restricted to the near-independent regime $\bm{I}_0$ (as has been the case throughout this section). Let $\bP$ denote the empirical measure of configurations $(\usi_{\delta v},\uvsi_{\delta v},\uL_{\delta v})$, taken over all variables $v$ of type $\bT$, that gives the maximal contribution to this restricted second moment. For $e\in\delta v$ let $\bzeta_e$ denote the marginal law of $(\varsigma_e,\bL_e)$ under $\bP$, and suppose it satisfies
	\begin{align}
	\label{e:ZETA.ww.unif.bound}
	1-\bzeta_e(\whi\whi\,|\,\bL)
	&\le \f{k^{O(1)}}{2^k}\,,\\
	\label{e:ZETA.hat.iota.unif.bound}
	\hit_{e,\bL}
	\equiv \bzeta_e\Big(\varsigma\in\set{\RYC}^2
		\,\Big|\,\bL\Big)
	+ \f{k^4}{4^k}
	&\le \f1{2^{k(1+\EPSINT)}}
	\end{align}
for all $e\in\delta v$ and all $\bL$. Then it must satisfy the estimate
	\beq\label{e:balanceMarg}
	\bigg|\f{\bP(\sigma_e=\sigma\,|\,\bL_e=\bL)}
		{\starpi_e(\sigma^1)
		\starpi_e(\sigma^2)
		}-1\bigg| \le \f{k^{O(1)}}{2^{k\EPSINT}}
	\eeq
for all $\sigma\in\set{\RYGB}^2\setminus\set{\red\red}$ and all $\bL$.

\begin{proof} As in the statement of the lemma, let us fix a variable $v$ of type $\bT$ that is non-defective. Write
	\beq\label{e:uX.notation}
	\uX_{\delta v}
	\equiv (X_e)_{e\in\delta v}
	\equiv
	\bigg(\Big(\sigma_e,\varsigma_e\Big)
	\bigg)_{e\in\delta v}\,.\eeq
Note that, by the rules of Definition~\ref{d:white.violet}, the configuration 
$\uX_{\delta v}$ also implicitly encodes $\ucol_{\delta v}$. In the augmented model where each edge $e=(av)$ is also labelled with the clause type $\bL_e\equiv\bL_a$, denote
	\beq\label{e:cal.X}
	\cX\equiv \cX_v
	\equiv (\uX_{\delta v},
		\uL_{\delta v})
	\equiv 
	\bigg(
	(X_e)_{e\in\delta v}, 
	(\bL_e)_{e\in\delta v}\bigg)
	\,.\eeq
Recall that $\bzeta_e$ denotes the marginal law of $(\bL_e,\varsigma_e)$ under $\bP$. We can write it as
	\[\bzeta_e(\bL,\varsigma)
	\equiv \bP\bigg(\bL_e=\bL, 
		\varsigma_e=\varsigma\bigg)
	= \pi_{\DD}(\bL\,|\,\bt_e)
	\bzeta_e(\varsigma\,|\,\bL)\,.\]
The measure $\bP$ must maximize entropy subject to the marginal constraints
	{\setlength{\jot}{0pt}
	\begin{alignat}{2}
	\label{e:final.D.L.clause}
	\bP(\bL_e=\bL)
		&=\pi_{\DD}(\bL\,|\,\bt_e)
		\quad&&\textup{for all $\bL$,}\\
	\label{e:final.D.L.judicious}
	\bP( (\sigma_e)^i=\sigma^i\,|\,\bL_e=\bL)
		&= \starpi_e(\sigma^i)
		\quad&&\textup{for all $\bL$
			and all $\sigma^i\in\set{\RYGB}$,}\\
	\label{e:final.D.L.varsigma}
	\bP(\varsigma_e=\varsigma\,|\,\bL_e=\bL)
		&= \bzeta_e(\varsigma\,|\,\bL)
		\quad&&\textup{for all $\bL$
		and all $\varsigma\in\set{\RYC,\whi}^2$,}
	\end{alignat}}%
for all $e\in\delta v$. By the method of Lagrange multipliers, $\bP$ must take the form
	\beq\label{e:final.D.L.Lagrangian.form}
	\bP(\usi_{\delta v},
	\uvsi_{\delta v},\uL_{\delta v})
	\cong
	\psi_v(\usi_{\delta v})
	\prod_{e\in\delta v}\Bigg\{
	\Ind{\sigma_e\sim\varsigma_e}
	\psi_e(\bL_e)
	\Bigg(
	\prod_{i=1,2}(\beta_e)^i((\sigma_e)^i\,|\,\bL_e)
	\Bigg)
	\chi_e(\varsigma_e\,|\,\bL_e)
	\Bigg\}	
	\eeq
where $\psi_v(\usi_{\delta v})$ is as in \eqref{e:non.compound.weight.functional.form} from Definition~\ref{d:param.weights.Augmented}, and we write $\sigma\sim\varsigma$ to indicate compatibility: formally, it means for both $i=1,2$ that we have $\sigma^i\sim\vsi^i$ in the sense that the following holds:
	{\setlength{\jot}{0pt}\begin{align*}
	\vsi^i=\red \quad &\textup{whenever $\sigma^i=\red$,}\\
	\vsi^i\in\set{\yel,\whi}
		\quad &\textup{whenever $\sigma^i=\yel$,}\\
	\vsi^i\in\set{\cya,\whi}
		\quad &\textup{whenever $\sigma^i\in\set{\grn,\blu}$.}
	\end{align*}}%
We now turn to the construction of the weights in \eqref{e:final.D.L.Lagrangian.form}. We divide the remainder of the argument into a few parts.\smallskip

\noindent\bemph{Part~1. Single-copy estimates.}
Let $\nu\equiv(\dbh,\hbh)$ be the vertex empirical measure that gives the maximal contribution to the first moment (of judicious configurations). The edge marginal of $\hbh_{\bL}$ is the canonical marginal $\omega_{\bL,j}=\starpi_{\bL(j)}$. Given the mapping from $\usi_{\delta a}$ to $\uvsi_{\delta a}$ within each clause, the measure $\hbh_{\bL}$ induces a measure on configurations $\uvsi_{\delta a}$. Let $\starzeta_{\bL,j}\equiv\starzeta_{\bL(j)}$ denote the marginal law of $\varsigma_j$ under this measure; we call this the canonical marginal on $\varsigma^i$ for an edge of type $\bt=\bL(j)$. If $e$ is an edge of type $\bt$ then we also write $\starzeta_e\equiv\starzeta_{\bt}$. Explicitly,
	\begin{align*}
	\starzeta_e(\red)
	&\cong \dqstar_e(\red)
		\Bigg\{
		\prod_{e'\in\delta v\setminus e}
		\dqstar_{e'}(\yel)
		\Bigg\}\,,\\
	\starzeta_e(\yel)
	&\cong 
	\dq_e(\yel)
	\Bigg\{
	\sum_{e'\in\delta v\setminus e}
		\dq_{e'}(\red)
	\prod_{e''\in\delta v\setminus \set{e,e'}}
	\dq_{e''}(\yel)\Bigg\}\,,\\
	\starzeta_e(\cya)
	&\cong 
	\dq_e(\cya)
	\Bigg\{
	\sum_{e'\in\delta v\setminus e}
		\dq_{e'}(\cya)
		\prod_{e''\in\delta v\setminus \set{e,e'}}
		\dq_{e''}(\yel)\Bigg\}\,,\\
	\starzeta_e(\whi)
	&\cong \dqstar_e(\set{\yel,\cya})
	\prod_{e'\in\delta v\setminus e}
	\dqstar_{e'}(\set{\yel,\cya})
	\Bigg( 1-\f{k^{O(1)}}{2^k}\Bigg)\,,
	\end{align*}
where the last estimate uses the assumption that $v$ is non-defective, hence nice, so that $\dqstar_e$ satisfies the estimates of Definition~\ref{d:nice}
for all $e\in\delta v$. It follows that
	\beq\label{e:ZETA.zetastar}
	\begin{pmatrix}
	\starzeta_e(\red)\\
	\starzeta_e(\yel)\\
	\starzeta_e(\cya)\\
	\starzeta_e(\whi)
	\end{pmatrix}
	= \begin{pmatrix}
	\Theta(1/2^k)\\
	\Theta(k/2^k)\\
	\Theta(k/2^k)\\
	1 - \Theta(k/2^k)
	\end{pmatrix}\,.\eeq
(We also remark that $\starzeta_e(\yel)=\starzeta_e(\cya)$, although we will not use this fact in what follows.)\smallskip

\noindent\bemph{Part~2. Single-copy weights.} We first consider the simpler problem of setting weights in the single-copy model. For this discussion, let $\sigma\in\set{\RYGB}$ and $\varsigma\in\set{\RYC,\whi}$, and again write $\sigma\sim\varsigma$ to indicate compatibility.
Given a probability measure $\starpi$ over $\set{\RYGB}$, along with a probability measure $\starzeta$ over spins $\set{\RYC,\whi}$ such that $\starzeta(\red)=\starpi(\red)$, we look for weights $\bstar$ and $\xstar$ such that
	\beq\label{e:mu.single.edge.sigma.varsigma}
	\mu(\sigma,\varsigma)
	\equiv
	\Ind{\sigma\sim\varsigma}
	\dqstar(\sigma)
	\bstar(\sigma)
	\xstar(\varsigma)
	\eeq
defines a probability measure over $(\sigma,\varsigma)$ whose marginal on $\sigma$ is $\starpi$, and whose marginal on $\varsigma$ is $\starzeta$. Explicitly, for the $\sigma$-marginal to be $\starpi$ we must have the equations
	{\setlength{\jot}{0pt}\begin{align*}
	\starpi(\red) &= \dqstar(\red) \bstar(\red)\xstar(\red)
		= \starzeta(\red)\,,\\
	\starpi(\yel) &= \dqstar(\yel)\bstar(\yel)[\xstar(\yel)
		+\xstar(\whi)]\,,\\
	\starpi(\grn) &=\dqstar(\grn)\bstar(\grn)[\xstar(\cya)
		+\xstar(\whi)]\,,\\
	\starpi(\blu) &=\dqstar(\blu)\bstar(\blu)[\xstar(\cya)
		+\xstar(\whi)]\,.\end{align*}
For the $\vsi$-marginal to be $\starzeta$ we must have the equations
	\begin{align*}\starzeta(\yel)
		&=\dqstar(\yel)\bstar(\yel)\xstar(\yel)\,,\\
	\starzeta(\cya)
		&=[\dqstar(\grn)\bstar(\grn) 
			+\dqstar(\blu)\bstar(\blu) ]
			\xstar(\cya)\,,\\
	\starzeta(\whi)
		&=[\dqstar(\yel)\bstar(\yel)
		+\dqstar(\grn)\bstar(\grn)
	+\dqstar(\blu)\bstar(\blu)
	]\xstar(\whi)\,.
	\end{align*}}%
Note also that if we multiply all the $\bstar$-weights by a scaling factor and divide all the $\xstar$-weights by the same factor, it has no effect on the right-hand side of \eqref{e:mu.single.edge.sigma.varsigma}, so without loss we can pin down the weights by requiring
$\bstar(\red)=1$ and $\xstar(\whi)=1$.
Then, combining the equations for $\starpi(\yel)$ and $\starzeta(\yel)$ gives
	\[
	\bstar(\yel) =\f{\starpi(\yel)-\starzeta(\yel)}{\dqstar(\yel)}\,,\quad
	\xstar(\yel) = \f{\starzeta(\yel)}{\starpi(\yel)-\starzeta(\yel)}\,.
	\]
Next, comparing the equations for $\starpi(\grn)$ and $\starpi(\blu)$ gives
	\[
	\f{\bstar(\grn)}{\bstar(\blu)}
	= \f{\starpi(\grn)/\dqstar(\grn)}{\starpi(\blu)/\dqstar(\blu)}
	= \f{\hqstar(\grn)}{\hqstar(\blu)}=1\,,
	\]
and we hereafter denote $\bstar(\cya)\equiv \bstar(\grn)\equiv\bstar(\blu)$. 
Combining with the equations for 
$\starpi(\grn)$, $\starpi(\blu)$, and $\starzeta(\cya)$ gives
	\[
	\bstar(\cya)=
	\f{\starpi(\grn)+\starpi(\blu)-\starzeta(\cya)}
		{\dqstar(\grn) + \dqstar(\blu)}
	= \f{\starpi(\cya)-\starzeta(\cya)}{\dqstar(\cya)}\,,\quad
	\xstar(\cya)
	=\f{\starzeta(\cya)}{\starpi(\cya)-\starzeta(\cya)}\,.
	\]
In summary, a valid solution is given by taking
	\begin{align}\label{e:ZETA.bstar}
	\begin{pmatrix}
	\bstar(\red)\\
	\bstar(\yel)\\
	\bstar(\grn)=\bstar(\blu)
	\equiv\bstar(\cya)
	\end{pmatrix}
	&= 
	\begin{pmatrix}
	1\\
	[\starpi(\yel)-\starzeta(\yel)]/\dqstar(\yel)\\
	[\starpi(\cya)-\starzeta(\cya)]/\dqstar(\cya)
	\end{pmatrix}
	\stackrel{\eqref{e:ZETA.zetastar}}{=} 
	\begin{pmatrix}
	\Theta(1)\\ \Theta(1) \\ \Theta(1)
	\end{pmatrix}\,,\\
	\label{e:ZETA.xstar}
	\begin{pmatrix}
	\xstar(\red)\\
	\xstar(\yel)\\
	\xstar(\cya)\\
	\xstar(\whi)
	\end{pmatrix}
	&=\begin{pmatrix}
	\starpi(\red)/\dqstar(\red)\\
	\starzeta(\yel)/[\starpi(\yel)-\starzeta(\yel)]\\
	\starzeta(\cya)/[\starpi(\cya)-\starzeta(\cya)]\\
	1
	\end{pmatrix}
	\stackrel{\eqref{e:ZETA.zetastar}}{=} 
	\begin{pmatrix}
	\Theta(1/2^k)\\
	\Theta(k/2^k)\\
	\Theta(k/2^k)\\ 1
	\end{pmatrix}
	\,.\end{align}
Substituting these into \eqref{e:mu.single.edge.sigma.varsigma}
gives a measure with the desired marginals $\starpi$ and $\starzeta$.\smallskip

\noindent\bemph{Part~3. Initialization in pair model.} Returning to \eqref{e:final.D.L.Lagrangian.form}, our goal is to construct a sequence of measures
	\[
	\bP_t(\usi_{\delta v},
	\uvsi_{\delta v},\uL_{\delta v})
	\cong
	\psi_{v,t}(\usi_{\delta v})
	\prod_{e\in\delta v}\Bigg\{
	\psi_{e,t}(\bL_e)
	\Bigg(
	\prod_{i=1,2}
	(\beta_{e,t})^i((\sigma_e)^i\,|\,\bL_e)
	\Bigg)
		\chi_{e,t}(\varsigma_e\,|\,\bL_e)
	\Bigg\}\]
which converges as $t\to\infty$ to the desired solution $\bP_\infty\equiv\bP$. We will further decompose the $\beta$ weights as
	\[
	(\beta_{e,t})^i( \sigma^i\,|\,\bL_e)
	\equiv
	\bigg\{(\dot{\beta}_{e,t})^i( \sigma^i\,|\,\bL_e)\bigg\}
	\cdot
	\bigg\{(\hat{\beta}_{e,t})^i( \sigma^i\,|\,\bL_e)\bigg\}
	\]
such that $(\hat{\beta}_{e,t})^i$ does not distinguish between $\grn$ and $\blu$, that is, such that
	\[
	(\hat{\beta}_{e,t})^i( \grn\,|\,\bL_e)
	=(\hat{\beta}_{e,t})^i( \blu\,|\,\bL_e)
	\equiv(\hat{\beta}_{e,t})^i( \cya\,|\,\bL_e)\,.
	\]
This is clearly an over-parametrization, so the $\dot{\beta}$ and $\hat{\beta}$ weights will not be uniquely determined; we need only find one choice of weights such that the resulting measure \eqref{e:final.D.L.Lagrangian.form} satisfies the constraints \eqref{e:final.D.L.clause}--\eqref{e:final.D.L.varsigma}. The overparametrization will be useful below because it allows for some separation between the analysis of the $\vsi$-marginals and the analysis of the $(\sigma,\bL)$-marginals.

We initialize the construction at $t=0$ as follows.
We first set $\psi_{v,0}(\usi_{\delta v})$ to be equal to $\varphi_v(\usi_{\delta v})$, which we recall from
\eqref{e:color.factor.model} is simply
the indicator of a valid pair coloring $\usi_{\delta v}$.
For an edge $e$ of type $\bt$, let $\starpi_e\equiv\starpi_{\bt}$ be the canonical marginal on $\sigma^i$, and let $\starzeta_e\equiv\starzeta_{\bt}$ be the canonical marginal on $\varsigma^i$, as discussed above.
Let $\bstar_e$ and $\xstar_e$ be the corresponding weights defined by \eqref{e:ZETA.bstar} and \eqref{e:ZETA.xstar}; in particular, we recall
from \eqref{e:ZETA.bstar} that $\bstar_e$
does not distinguish between $\grn$ and $\blu$. We then set
	{\setlength{\jot}{0pt}\begin{alignat*}{2}
	\psi_{e,0}(\bL)
		&\equiv\pi_{\DD}(\bL\,|\,\bt_e)
		\quad
		&&\textup{for all $\bL$,}\\
	(\dot{\beta}_{e,0})^i(\sigma^i\,|\,\bL)
		&\equiv 1
		&&\textup{for
		$i=1,2$,
		all $\bL$,
		and all $\sigma^i\in\set{\RYGB}$,}\\
	(\hat{\beta}_{e,0})^i(\sigma^i\,|\,\bL)
		&\equiv \bstar_e(\sigma^i)\quad
		&&\textup{for
		$i=1,2$,
		all $\bL$,
		and all $\sigma^i\in\set{\RYGB}$,}\\
	\chi_{e,0}(\varsigma\,|\,\bL)
		&\equiv 
		\xstar_e(\varsigma^1)
		\xstar_e(\varsigma^2)\quad
		&&\textup{for all 
		$\bL$ and all $\varsigma\in\set{\RYC,\whi}^2$.}
	\end{alignat*}}%
Thus at $t=0$ the measure $\bP_0$ will satisfy constraints 
\eqref{e:final.D.L.clause} and \eqref{e:final.D.L.judicious}, but not
\eqref{e:final.D.L.varsigma}, since we will have
	\[
	\bzeta_{e,0}(\varsigma\,|\,\bL)
	\equiv
	\bP_0(\varsigma_e=\varsigma\,|\,\bL_e=\bL)
	=\prod_{i=1,2} \starzeta_e(\varsigma^i)\,,
	\]
which in general is not the same as
$\bzeta_e(\varsigma\,|\,\bL)$.\smallskip

\noindent\bemph{Part~4. Update procedure.}
In this step we will make use of Lemma~\ref{l:single.edge.sigma.varsigma}, which is stated and proved below. For $t\ge0$ let $\delta_{e,t}$ and $\ddot{\delta}_{e,t}$ denote parameters such that the estimates \eqref{e:ZETA.delta.error}--\eqref{e:ZETA.delta.error.mult.rr} below hold for the error between $\bzeta_{e,t}(\cdot\,|\,\bL)$ versus $\bzeta_e(\cdot\,|\,\bL)$ --- that is to say,
for all clause types $\bL$ that can appear incident to edge $e$, we assume that
	\begin{align}
	\label{e:ZETA.delta.error.t}
	\max
	\Bigg\{
	\Big|\bzeta_e(\varsigma\,|\,\bL)
		-\bzeta_{e,t}(\varsigma\,|\,\bL)\Big|
	: \varsigma\in\set{
	\red\whi,\yel\whi,\cya\whi,
	\whi\red,\whi\yel,\whi\cya}
	\Bigg\}
	&\le \f{k^{O(1)}\delta_{e,t}}{2^k}\,,\\
	\label{e:ZETA.delta.error.hat.iota.t}
	\max
	\Bigg\{
	\Big|\bzeta_e(\varsigma\,|\,\bL)
		-\bzeta_{e,t}(\varsigma\,|\,\bL)\Big|
	: \varsigma\in\set{\RYC}^2
	\Bigg\}
	&\le \hit \delta_{e,t}\,, \\
	\label{e:ZETA.delta.error.mult.rr.t}
	\bigg|
		\f{\bzeta_e(\red\red\,|\,\bL)}
		{\bzeta_{e,t}(\red\red\,|\,\bL)}-1\bigg|
	&\le\ddot{\delta}_{e,t}\,.
	\end{align}
At the initialization $t=0$ 
it follows from \eqref{e:ZETA.ww.unif.bound} and \eqref{e:ZETA.hat.iota.unif.bound} that
$\delta_{e,t} \le O(1)$
and $\ddot{\delta}_{e,t} \le O(2^{k(1-\EPSINT)})$.
Denote
	\beq\label{e:ZETA.defn.delta.e.t}
	\bde_t
	\equiv\sum_{e\in\delta v}\delta_{e,t}\,,\quad
	\ddbde_t
	\equiv\sum_{e\in\delta v}
		\min\set{\ddot{\delta}_{e,t},1}\,.\quad
	\eeq
Let $t\ge0$ be an integer time, and suppose inductively that we have constructed the weights at time $t$. The marginal law on an edge $e\in\delta v$ is given by
	\[
	\bP_t(\sigma_e=\sigma,\varsigma_e=\varsigma,\bL_e=\bL)
	\cong
	\Ind{\sigma\sim\varsigma}
	\dq_{e,t}(\sigma)
	\psi_{e,t}(\bL)
	\Bigg(\prod_{i=1,2} (\beta_{e,t})^i(\sigma^i\,|\,\bL)\Bigg)
	\chi_{e,t}(\varsigma\,|\,\bL)\,,
	\]
where $\dq_{e,t}$ is the probability measure over $\sigma\in\set{\RYC}^2$ defined by 
	\[
	\dq_{e,t}(\sigma)
	\cong
	\sum_{\usi_{\delta v}:
		\sigma_e=\sigma}
	\psi_{v,t}(\usi)
	\prod_{e'\in\delta v\setminus e}
	\Bigg\{
	\sum_{\bL_{e'}}
	\psi_{e',t}(\bL_{e'})
	\Bigg(\prod_{i=1,2}
	\beta_{e',t}((\sigma_{e'})^i\,|\,\bL_{e'})
	\Bigg)
	\sum_{\varsigma_{e'}:
		\varsigma_{e'} \sim \sigma_{e'}}
	\chi_{e',t}(\varsigma_{e'}\,|\,\bL_{e'})
	\Bigg\}\,.
	\]
The conditional law of $(\sigma_e,\varsigma_e)$ given $\bL_e$ is then given by
	\[
	\bP_t(\sigma_e=\sigma,\varsigma_e=\varsigma\,|\,
		\bL_e=\bL)
	\cong \Ind{\sigma\sim\varsigma}
	\dq_{e,t}(\sigma)
	\Bigg(\prod_{i=1,2}
		(\beta_{e,t})^i(\sigma^i\,|\,\bL)\Bigg)
	\chi_{e,t}(\varsigma\,|\,\bL)\,.
	\]
For each $\bL$, we apply Lemma~\ref{l:single.edge.sigma.varsigma} (below) to find updated weights $(\hat{\beta}_{e,t+3/4})^i(\sigma^i\,|\,\bL)$ and $\chi_{e,t+3/4}(\varsigma\,|\,\bL)$ such that, if we define $(\beta_{e,t+3/4})^i
\equiv (\dot{\beta}_{e,t})^i \cdot (\hat{\beta}_{e,t+3/4})^i$, then 
 the probability measure
	\[
	\tilde{\mu}_{e,t+3/4}
	(\sigma,\varsigma\,|\,\bL)
	\cong
	\Ind{\sigma\sim\varsigma}
	\dq_{e,t}(\sigma)
	\Bigg(\prod_{i=1,2}
	(\beta_{e,t+3/4})^i(\sigma^i\,|\,\bL) 
	\Bigg)
	\chi_{e,t+3/4}(\varsigma\,|\,\bL)
	\]
has $\varsigma$-marginal exactly $\bzeta_e(\cdot\,|\,\bL)$. The marginal law of $(\usi_{\delta v},\uL_{\delta v})$ at time $t+3/4$ is then given by
	\[
	\bP_{t+3/4}(\usi_{\delta v},\uL_{\delta v})
	\cong
	\psi_{v,t}(\usi_{\delta v})
	\prod_{e\in\delta v}\Bigg\{
	\underbrace{
	\psi_{e,t}(\bL_e)
	\sum_{\varsigma_e:\varsigma_e
		\sim \sigma_e}
	\Bigg(\prod_{i=1,2}
		(\beta_{e,t+3/4})^i((\sigma_e)^i\,|\,\bL_e)
			\Bigg)
	\chi_{e,t+3/4}(\varsigma_e\,|\,\bL_e)
	}_{\cong \hq_{e,t+3/4}(\sigma_e,\bL_e)}
	\Bigg\}\,,
	\]
where $\hq_{e,t}$ is a probability measure over pairs $(\sigma,\bL)$. It follows from Lemma~\ref{l:single.edge.sigma.varsigma} that
	\beq\label{e:ZETA.hq.first.error}
	\CRELERR(\hq_{e,t},\hq_{e,t+3/4})
	\le
	\max_{\bL}\left\{
	k^{O(1)}
	\begin{pmatrix}
	\hit_{e,\bL}&0\\
	2^k \hit_{e,\bL}&0\\
	1&1
	\end{pmatrix}
	\begin{pmatrix}
	\delta_{e,t}\\
	\ddot{\delta}_{e,t}
	\end{pmatrix}
	\right\}
	\stackrel{\eqref{e:ZETA.hat.iota.unif.bound}}{\le}
	k^{O(1)}
	\begin{pmatrix}
	2^{-k(1+\EPSINT)}&0\\
	2^{-k\EPSINT}&0\\
	1&1
	\end{pmatrix}
	\begin{pmatrix}
	\delta_{e,t}\\
	\ddot{\delta}_{e,t}
	\end{pmatrix}
	\,.\eeq
Then apply Proposition~\ref{p:ctypes.varupdate} (where the $\hat{p}_e$ and $\hq_e$ of 
Proposition~\ref{p:ctypes.varupdate} are given by
$\hq_{e,t}$ and $\hq_{e,t+3/4}$ respectively) to find new weights $\psi_{v,t+1}(\usi)$, $\psi_{e,t+1}(\bL)$, and $(\dot{\beta}_{e,t+1})^i(\sigma^i\,|\,\bL)$ such that, if we define
$(\beta_{e,t+1})^i \equiv (\dot{\beta}_{e,t+1})^i(\hat{\beta}_{e,t+3/4})^i$,
then the probability measure
	\[
	\bP_{t+1}(\usi_{\delta v},\uvsi_{\delta v},\uL_{\delta v})
	\cong
	\psi_{v,t+1}(\usi_{\delta v})
	\prod_{e\in\delta v}\Bigg\{
	\psi_{e,t+1}(\bL_e)
	\Bigg(
	\prod_{i=1,2}
	(\beta_{e,t+1})^i((\sigma_e)^i\,|\,\bL_e)
	\Bigg)
		\chi_{e,t+1}(\varsigma_e\,|\,\bL_e)
	\Bigg\}\,,
	\]
has a marginal on $(\usi_{\delta v},\bL_{\delta v})$ that is fully judicious
(in the sense of Definition~\ref{d:judicious.augmented.alphabet}). In the notation of Proposition~\ref{p:ctypes.varupdate}, in going between $\hq_{e,t}$ and $\hq_{e,t+3/4}$ we have
	\[
	\err_v
	= \err_{v,t}
	= \sum_{e\in\delta v}
	\f{\delta_{e,t} + \min\set{\ddot{\delta}_{e,t},1}}
		{2^{k(1+\EPSINT)}}
	= \f{\bde_t + \ddbde_t}{2^{k(1+\EPSINT)}}\,,
	\]
using the notation of \eqref{e:ZETA.defn.delta.e.t}.
It then follows from the bound \eqref{e:ctypes.varupdate.MAINBOUND} of 
Proposition~\ref{p:ctypes.varupdate} that
	\begin{align}\nonumber
	\max_{\bL}\Bigg\{
	\bigg|\f{\psi_{e,t+1}(\bL)}{\psi_{e,t+3/4}(\bL)}
	-1\bigg|
	+ \max_{\tau\ne\red}
	\bigg|\f{(\beta_{e,t+1})^i(\tau\,|\,\bL)}
	{(\beta_{e,t+3/4})^i(\tau\,|\,\bL)}-1\bigg|
	\Bigg\}
	&\le \f{k^{O(1)}}{2^{k(1+\EPSINT)}}
		\bigg(\delta_{e,t} + 
			\min\set{\ddot{\delta}_{e,t},1}
			+ \err_{v,t}\bigg)
	\,,\\
	\max_{\bL}\Bigg\{
	\bigg|\f{(\beta_{e,t+1})^i(\red\,|\,\bL)}
	{(\beta_{e,t+3/4})^i(\red\,|\,\bL)}-1\bigg|
	\Bigg\}
	&\le \f{k^{O(1)}}{2^{k\EPSINT}}
		\bigg(\delta_{e,t} + 
			\min\set{\ddot{\delta}_{e,t},1}
			+ \err_{v,t}\bigg)
			\,.
	\label{e:ZETA.beta.last.error}
	\end{align}
It follows from the bound \eqref{e:ctypes.varupdate.OUTMSG} of Proposition~\ref{p:ctypes.varupdate} that
	\begin{align}\nonumber
	\VRELERR(\dq_{e,t},\dq_{e,t+1})
	&\le k^{O(1)}
	\begin{pmatrix}
	1\\1\\1\\ 
	2^{-k\EPSINT} \\
	2^{-k\EPSINT}
	\end{pmatrix}
		\err_{v,t}
	+k^{O(1)}
	\begin{pmatrix}
	0&0&0\\
	0&1&1\\
	0&1&1\\
	1&2^{-k}&2^{-k}\\
	1&1&1\\
	\end{pmatrix}
	\begin{pmatrix}
	2^{-k(1+\EPSINT)}\delta_{e,t}\\
	2^{-k\EPSINT}\delta_{e,t}\\
	2^{-k\EPSINT} \min\set{ \delta_{e,t}
		+ \ddot{\delta}_{e,t},1} 
	\end{pmatrix} \\
	&\le
	k^{O(1)}
	\begin{pmatrix}
	1\\1\\1\\ 
	2^{-k\EPSINT} \\
	2^{-k\EPSINT}
	\end{pmatrix} 
	\f{\bde_t+\ddbde_t
		}{2^{k(1+\EPSINT)}}
	+k^{O(1)}
	\begin{pmatrix}
	0\\1\\1\\ 0\\1
	\end{pmatrix}
	\f{\delta_{e,t}
		+\min\set{\ddot{\delta}_{e,t},1}}
		{2^{k\EPSINT}}\,.
	\label{e:ZETA.qdot.error}
	\end{align}
At time $t+1$ we have
	\[
	\f{\bzeta_e(\varsigma\,|\,\bL)}
		{\bzeta_{e,t+1}(\varsigma\,|\,\bL)}
	=
	\f{\DS
	z_{e,t+1}\sum_{\sigma:\sigma\sim\varsigma}
	\dq_{e,t}(\sigma)
	(\beta_{e,t+3/4})^1(\sigma^1)
	(\beta_{e,t+3/4})^2(\sigma^2)
	 \chi_{e,t+3/4}(\varsigma\,|\,\bL)}
	{\DS
	\tilde{z}_{e,t+3/4}
	\sum_{\sigma:\sigma\sim\varsigma}
	\dq_{e,t+1}(\sigma)
	(\beta_{e,t+1})^1(\sigma^1)
	(\beta_{e,t+1})^2(\sigma^2)
	 \chi_{e,t+3/4}(\varsigma\,|\,\bL)
	}
	\]
for normalizing constants $\tilde{z}_{e,t+3/4}$ and $z_{e,t+1}$. Combining with \eqref{e:ZETA.beta.last.error} and \eqref{e:ZETA.qdot.error} then gives
	\[
	\delta_{e,t+1}
	+\ddot{\delta}_{e,t+1}
	\le
	\f{k^{O(1)}}{2^{k\EPSINT}}\Bigg(
	\f{\bde_t+\ddbde_t}{2^k}
	+ \delta_{e,t} 
	+ \min\set{\ddot{\delta}_{e,t},1}
	\Bigg)\,.
	\]
This concludes our analysis of the update procedure.\smallskip

\noindent\bemph{Part~5. Conclusion.}
Summing the preceding bound over all $e\in\delta v$ gives
	\begin{align*}
	\sum_{t\ge0} (\bde_t+\ddbde_t)
	&\le \bde_0 + \ddbde_0
	+ 
	\f{k^{O(1)}}{2^{k\EPSINT}}
	\sum_{t\ge0} (\bde_t+\ddbde_t)
	\le O(\bde_0 + \ddbde_0)
	\le O(k 2^k)\,,\\
	\sum_{t\ge0} 
	(\delta_{e,t} + \min\set{\ddot{\delta}_{e,t},1})
	&\le 
	\delta_{e,0} + \ddot{\delta}_{e,0}
		+ \f{k^{O(1)}}{2^{k\EPSINT}}
		 \sum_{t\ge0}
		\Bigg( \f{\bde_t+\ddbde_t}{2^k}+
		\delta_{e,t} + \min\set{\ddot{\delta}_{e,t},1}
		\Bigg)
	\le O(1)\,.
	\end{align*}
We can use the above bounds with \eqref{e:ZETA.beta.last.error} to obtain
	\[
	\sum_{t\ge0} \CRELERR(\hq_{e,t+3/4},\hq_{e,t+1})
	\le \begin{pmatrix}
	2^{-k} \\ 1 \\ 1
	\end{pmatrix}
	\f{k^{O(1)}}{2^{k\EPSINT}}
	\sum_{t\ge0}
	\bigg(\delta_{e,t} + 
			\min\set{\ddot{\delta}_{e,t},1}
			+ \err_{v,t}\bigg)
	\le \begin{pmatrix}
	2^{-k} \\ 1 \\ 1
	\end{pmatrix}
	\f{k^{O(1)} }{2^{k\EPSINT}}\,.
	\]	
Combining with \eqref{e:ZETA.hq.first.error} gives
	\[
	\CRELERR(\hq_{e,0},\hq_{e,\infty})
	\le O(1)
	\sum_{t\ge0}\Bigg\{
	\CRELERR(\hq_{e,t},\hq_{e,t+3/4})
	+\CRELERR(\hq_{e,t+3/4},\hq_{e,t+1})
	\Bigg\}
	\le
	\f{k^{O(1)}}{2^{k\EPSINT}}
	\begin{pmatrix}
	2^{-k}\\
	1\\
	2^k 
	\end{pmatrix}\,.
	\]
(We also see that 
condition~\eqref{e:condition.on.hat.msg.with.L}
is satisfied by $\hq_{e,t}$ for all $t\ge0$, so that 
the above applications of Proposition~\ref{p:ctypes.varupdate} are justified.)
We also obtain from \eqref{e:ZETA.qdot.error} that
	\[
	\VRELERR(\dq_{e,0},\dq_{e,\infty})
	\le
	\begin{pmatrix}
	1\\1\\1\\ 
	2^{-k\EPSINT} \\ 1
	\end{pmatrix} 
	\f{k^{O(1)}}{2^{k\EPSINT}}\,.
	\]
The above estimates imply that for all $\sigma\in\set{\RYGB}^2\setminus\set{\red\red}$ we have
	\[\bP(\sigma_e=\sigma\,|\,\bL_e=\bL)
	\cong \dq_{e,\infty}(\sigma\,|\,\bL)
		\hq_{e,\infty}(\sigma\,|\,\bL)
	= \dqstar_e(\sigma) \hqstar_e(\sigma)
		\Bigg( 1 + \f{k^{O(1)}}{2^{k\EPSINT}}\Bigg)
	= \prodom_e(\sigma)
		\Bigg( 1 + \f{k^{O(1)}}{2^{k\EPSINT}}\Bigg)\,.
	\]
This implies the claimed bounds \eqref{e:balanceMarg}.
\end{proof}
\end{lem}

\begin{lem}[used only in proof of Lemma~\ref{l:ZETA.conditions.balanced.vertex}]
\label{l:single.edge.sigma.varsigma} In this lemma we take $\sigma$ in the reduced alphabet $\set{\RYC}^2$. We continue to take $\vsi\in\set{\RYC,\whi}^2$. We write $\sigma\sim\vsi$ to indicate compatibility: if $\sigma^i=\red$ then $\vsi^i=\red$; if $\sigma^i\in\set{\yel,\cya}$ then $\vsi^i\in\set{\sigma^i,\whi}$. On edge $e$ let $\mu_0$ be a probability measure on pairs $(\sigma,\varsigma)$ with $\sigma\sim\varsigma$, of the form
	\[
	\mu_0(\sigma,\varsigma)
	=\f{\Ind{\sigma\sim\varsigma}}{z_0}
	\dq(\sigma) 
	(b_0)^1(\sigma^1)(b_0)^2(\sigma^2)
	x_0(\varsigma)\,,
	\]
such that $\mu(\sigma^i=\tau)=\starpi_e(\tau)$ for both $i=1,2$ and all $\tau\in\set{\RYC}$. Assume that $\starpi_e$ is nice in the sense of Definition~\ref{d:nice}, that $\dq(\sigma)=\Theta(1)$ for all $\sigma\in\set{\RYC}^2$, and that
$b^i(\tau)=\Theta(1)$ for $i=1,2$ and all $\tau\in\set{\RYC}$. Let $\zeta_0$ denote the marginal law of $\varsigma$ under $\mu_0$. Let $\zeta$ be another probability measure over $\varsigma\in\set{\RYC,\whi}^2$ whose single-copy marginals satisfy
$\zeta^1(\red)=\starpi_e(\red)=\zeta^2(\red)$.
Assume that 
	\begin{align}\label{e:ZETA.assume.ww}
	\max\bigg\{ 1-\zeta_0(\whi\whi),
	1-\zeta(\whi\whi)\bigg\}
	&\le \f{k^{O(1)}}{2^k}\,,\\
	\hit \equiv
	\f{k^4}{4^k} + \max\bigg\{
	\zeta_0(\set{\RYC}^2),
	\zeta(\set{\RYC}^2)
	\bigg\}
	&\le \f1{2^{k(1+\EPSINT)}}\,.
	\label{e:ZETA.assume.iota}
	\end{align}
Assume moreover that we have parameters
$\delta \le O(1)$ and $\ddot{\delta}\le 2^{k(1-\EPSINT)}$ such that
	\begin{align}
	\label{e:ZETA.delta.error}
	\max
	\Bigg\{
	\Big|\zeta(\varsigma)-\zeta_0(\varsigma)\Big|
	: \varsigma\in\set{
	\red\whi,\yel\whi,\cya\whi,
	\whi\red,\whi\yel,\whi\cya}
	\Bigg\}
	&\le \f{k^{O(1)}\delta}{2^k}\,,\\
	\label{e:ZETA.delta.error.hat.iota}
	\max
	\Bigg\{
	\Big|\zeta(\varsigma)-\zeta_0(\varsigma)\Big|
	: \varsigma\in\set{\RYC}^2
	\Bigg\}
	&\le \hit \delta\,, \\
	\label{e:ZETA.delta.error.mult.rr}
	\bigg|
		\f{\zeta(\red\red)}
		{\zeta_0(\red\red)}-1\bigg|
	&\le\ddot{\delta}\,.
	\end{align}
Then there exist weights $(b_{3/4})^1$, $(b_{3/4})^2$, $x_{3/4}$ such that under the corresponding measure $\mu_{3/4}$ (see \eqref{e:ZETA.mu.t}) the spin $\varsigma$ has marginal $\zeta_{3/4}=\zeta$. Moreover, if we let $\hat{z}_t$ be the normalizing constant such that
	\[
	\hq_t(\sigma)
	= \f1{\hat{z}_t}
	(b_t)^1(\sigma^1)(b_t)^2(\sigma^2)
	\sum_{\varsigma:\varsigma\sim\sigma}
	x_t(\varsigma)
	\]
is a probability measure,
then the new weights can be chosen such that the error between $\hq_0$ and $\hq_{3/4}$ is very small:
	\[
	\crelerr(\hq_0,\hq_{3/4})
	\le k^{O(1)} \begin{pmatrix}
	\hit & 0 \\ 2^k\hit &0 \\ 1& 1
	\end{pmatrix}
	\begin{pmatrix}
	\delta \\ \ddot{\delta}
	\end{pmatrix}\,,\]
using the notation of 
Definition~\ref{d:clause.rel.error}.

\begin{proof} We begin with some easy observations.
By the assumption that $\dq=\Theta(1)$ 
and $b^i=\Theta(1)$, we have
	\[
	\mu_0(\sigma,\varsigma)
	=
	\f{\Ind{\sigma\sim\varsigma} }{z_0}
	\Theta(1)
	x(\varsigma)\,.
	\]
It then follows from \eqref{e:ZETA.assume.ww}
and \eqref{e:ZETA.assume.iota} that
$x(\whi\whi)/z_0\ge 1-k^{O(1)}/2^k$, while
	\begin{align*}
	\max\Bigg\{
	\f{x(\varsigma)}{z_0}
	: \varsigma\in\set{\RYC,\whi}^2
		\setminus\set{\whi\whi}
	\Bigg\}
	&\le
	\f{k^{O(1)}}{2^k}\,,\\
	\max\Bigg\{
	\f{x(\varsigma)}{z_0}
	: \varsigma\in\set{\RYC}^2
	\Bigg\} 
	\le O(\hit)
	&\le \f{O(1)}{2^{k(1+\EPSINT)}}\,.
	\end{align*}
This implies that $\mu_0(\sigma)=\Theta(1)$ for all $\sigma\in\set{\yel,\cya}^2$, since in this case
	\[
	\mu_0(\sigma)
	=\f1{z_0}\bigg\{ \Theta(1) 
	x(\whi\whi) + \f{k^{O(1)}}{2^k}\bigg\}
	=\Theta(1)\,.
	\]
Next, by the assumption that $\starpi_e$ is nice and $\mu_0$ is judicious, we have
	\[
	\f{\Theta(1)}{2^k}
	= \starpi_e(\red)
	= \mu_0(\sigma^1=\red)
	= \f{\Theta(1)}{z_0}
	\sum_{\vartheta\in \set{\red}\times\set{\RYC,\whi}}
	x(\vartheta)
	= \f{\Theta(1)x(\red\whi)}{z_0}
	+ \f{O(1)}{2^{k(1+\EPSINT)}}\,,
	\]
which shows that $x(\red\whi)=\Theta(2^{-k})$. It follows from this that 
	\[
	\mu_0(\sigma=\vartheta)
	= \f{\Theta(1)}{2^k}
	\quad\textup{for all }
	\vartheta\in\set{\red\yel,\red\cya,
		\yel\red,\cya\red}\,.
	\]
Lastly, we note that \eqref{e:ZETA.assume.iota} directly implies
$\mu_0(\sigma=\red\red)
=\zeta_0(\red\red) \le 2^{-k(1+\EPSINT)}$. In summary we have seen that the assumptions of the lemma imply
	\beq\label{e:ZETA.mu.zero.estimates}
	\mu_0(\sigma=\vartheta)
	= \begin{cases}
	\Theta(1) & \textup{if $\red[\vartheta]=0$,}\\
	\Theta(2^{-k}) & \textup{if $\red[\vartheta]=1$,}\\
	O(2^{-k(1+\EPSINT)})
	& \textup{if $\red[\vartheta]=2$,}
	\end{cases}\eeq
for $\vartheta\in\set{\RYC}^2$.\smallskip

\noindent\bemph{Part 1. Description of update procedure.} We will define a sequence of weights $(b_t)^i$, $x_t$ ending with the desired weights at $t= 3/4$. For each $t$ denote the corresponding measure
	\beq\label{e:ZETA.mu.t}
	\mu_t(\sigma,\varsigma)
	=\f{\Ind{\sigma\sim\varsigma}}{z_t}
	\dq(\sigma) 
	(b_t)^1(\sigma^1)(b_t)^2(\sigma^2)
	x_t(\varsigma)\,,\eeq
with $z_t$ the normalizing constant. 
Let $\zeta_t$ denote the marginal law of $\varsigma$ under $\mu_t$.
\begin{enumerate}[(a)]
\item \label{i:ZETA.update.yw}
\textit{Time $t=0$ to $t=1/4$.} First adjust the weight on $\varsigma=\yel\whi$: set
	\begin{align}
	\label{e:ZETA.update.yw.b}
	\f{(b_{1/4})^1(\yel)}{(b_0)^1(\yel)}
	&= \f{\mu_0(\sigma^1=\yel)-\zeta(\yel\whi)}
	{\mu_0(\sigma^1=\yel)-\mu_0(\varsigma=\yel\whi)}
	\,,\\
	\f{x_{1/4}(\yel\whi)}{x_0(\yel\whi)}
	&= \f{\mu_0(\sigma^1=\yel)
	-\mu_0(\varsigma=\yel\whi)}
		{\mu_0(\sigma^1=\yel)-\zeta(\yel\whi)}
	\f{\zeta(\yel\whi)}{\mu_0(\varsigma=\yel\whi)}\,.
	\label{e:ZETA.update.yw.x}
	\end{align}
We call \eqref{e:ZETA.update.yw.b} and \eqref{e:ZETA.update.yw.x} the \bemph{$\yel\whi$-update}. At the same time also make the updates
\eqref{e:ZETA.update.yw.b}
and \eqref{e:ZETA.update.yw.x}
with $\cya$ in place of $\yel$,
to define new weights
$(b_{1/4})^1(\cya)$ and $x_{1/4}(\cya\whi)$
--- we call this the \bemph{$\cya\whi$-update}.
Keep all the other weights unchanged, that is,
$(b_{1/4})^1(\red)=(b_0)^1(\red)$,
$(b_{1/4})^2(\tau)=(b_0)^1(\tau)$ for all $\tau\in\set{\RYC}$, and $x_{1/4}(\varsigma)=x_0(\varsigma)$ for all $\varsigma \notin\set{\yel\whi,\cya\whi}$.

\item \label{i:ZETA.update.wy}
Time $t=1/4$ to $t=1/2$. Perform the analogous update in the second copy: for the \bemph{$\whi\yel$-update}, set
	\begin{align*}
	\f{(b_{1/2})^2(\yel)}{(b_{1/4})^2(\yel)}
	&= \f{\mu_{1/4}(\sigma^2=\yel)-\zeta(\whi\yel)}
	{\mu_{1/4}(\sigma^2=\yel)
	-\mu_{1/4}(\varsigma=\whi\yel)}
	\,,\\
	\f{x_{1/2}(\whi\yel)}{x_{1/4}(\whi\yel)}
	&= \f{\mu_{1/4}(\sigma^2=\yel)
	-\mu_{1/4}(\varsigma=\whi\yel)}
		{\mu_{1/4}(\sigma^2=\yel)-\zeta(\whi\yel)}
	\f{\zeta(\whi\yel)}
	{\mu_{1/4}(\varsigma=\whi\yel)}\,.
	\end{align*}
Likewise make the \bemph{$\whi\cya$-update}, to define new weights $(b_{1/2})^2(\cya)$ and $x_{1/2}(\whi\cya)$. Again, keep all the other weights unchanged.

\item \label{i:ZETA.update.x.only} Time $t=1/2$ to $t=3/4$. Finally, for all $\varsigma\in\set{\RYC,\whi}^2$, update
	\[
	\f{x_{3/4}(\varsigma)}{x_{1/2}(\varsigma)}
	= \f{\zeta(\varsigma)}{\zeta_{1/2}(\varsigma)}\,.
	\]
Leave the $b$-weights unchanged, that is,
$(b_{3/4})^i(\tau)=(b_{1/2})^i(\tau)$
for $i=1,2$ and all $\tau\in\set{\RYC}$.
\end{enumerate}
We begin with a simple observation about the first update, from $t=0$ to $t=1/4$. At time $t=0$, the non-normalized weight on the event $\set{\sigma^1=\yel}$ is given by $z_0\mu_0(\sigma^1=\yel)$. At time $t=1/4$, the non-normalized weight on the same event is
	\begin{align}\nonumber
	z_{1/4}\mu_{1/4}(\sigma^1=\yel)
	&= z_0\f{(b_{1/4})^1(\yel)}{(b_0)^1(\yel)}
	\Bigg(
	\mu_0(\sigma^1=\yel,\varsigma\ne\yel\whi)
	+\mu_0(\varsigma=\yel\whi)
	\f{x_{1/4}(\yel\whi)}{x_0(\yel\whi)}
	\Bigg)
	\\
	&= 
	z_0 \Bigg\{
	\Big(\mu_0(\sigma^1=\yel)-\zeta(\yel\whi)\Big)
	+ \zeta(\yel\whi)
	 \Bigg\}
	= z_0 \mu_0(\sigma^1=\yel)\,,
	\label{e:ZETA.same.normalization}
	\end{align}
and likewise
$z_{1/4}\mu_{1/4}(\sigma^1=\cya)=
z_0 \mu_0(\sigma^1=\cya)$. The weight of the event $\set{\sigma^1=\red}$ remains unchanged, so altogether we have $z_{1/4}=z_0$. For the same reason $z_{1/2}=z_{1/4}$. We now turn to analyzing the effect of these updates on the $\sigma$-marginal.\smallskip

\noindent\bemph{Part~2. Effect of update \eqref{i:ZETA.update.yw} on $\sigma$-marginal.}
The first change in the $b$-weights can be bounded as
	\begin{align}\nonumber
	\bigg|\f{(b_{1/4})^1(\yel)}{(b_0)^1(\yel)}-1\bigg|
	&\stackrel{\eqref{e:ZETA.update.yw.b}}{=}
	\bigg|\f{\mu_0(\sigma^1=\yel)-\zeta(\yel\whi)}
	{\mu_0(\sigma^1=\yel)-\mu_0(\varsigma=\yel\whi)}-1\bigg|
	\stackrel{\eqref{e:ZETA.delta.error}}{=}
	\bigg|\f{\mu_0(\sigma^1=\yel)-\zeta(\yel\whi)}
	{\mu_0(\sigma^1=\yel)-\zeta(\yel\whi)
		- k^{O(1)} \delta/2^k}-1\bigg|\\
	&\stackrel{\eqref{e:ZETA.assume.ww}}{=}
	\bigg|\f{\mu_0(\sigma^1=\yel)
		- k^{O(1)}\delta/2^k}
	{\mu_0(\sigma^1=\yel)
		- k^{O(1)} \delta/2^k}-1\bigg|
	\le \f{k^{O(1)} \delta}{2^k}
	\label{e:ZETA.b.x.error}
	\end{align}
--- the last estimate above uses the assumption that $\starpi_e$ is nice, which implies that $\starpi_e(\yel)=\Theta(1)$.
Now, note it follows from \eqref{e:ZETA.same.normalization} that it does not make a difference if the $\yel\whi$-update and $\cya\whi$-update are done simultaneously or in sequence. Suppose for this part of the proof that we make only the $\yel\whi$-update (i.e., only \eqref{e:ZETA.update.yw.b} and \eqref{e:ZETA.update.yw.x}), without making the $\cya\whi$-update. Let $\mu_{1/8}$ denote the measure that results, and define the signed measure $\Gamma\equiv\Gamma_{1/8}\equiv\mu_{1/8}-\mu_0$; we now proceed to bound $\Gamma$. The event $\set{\varsigma=\yel\whi}$ is affected by both the $b$- and $x$-updates: 
	\begin{align}\nonumber
	\Big|\Gamma(\varsigma=\yel\whi)\Big|
	&\equiv\Big|\zeta_{1/8}(\yel\whi)
		-\zeta_0(\yel\whi)\Big|
	=\zeta_0(\yel\whi)
	\Bigg| 
	\f{(b_{1/4})^1(\yel)}{(b_0)^1(\yel)}
	\f{x_{1/4}(\yel\whi)}{x_0(\yel\whi)}-1\Bigg|
	\stackrel{\eqref{e:ZETA.update.yw.x}}{=}
	\Bigg| \zeta(\yel\whi)
	\f{(b_{1/4})^1(\yel)}{(b_0)^1(\yel)}
	-\zeta_0(\yel\whi)\Bigg|
	\\ \nonumber
	&\stackrel{\eqref{e:ZETA.update.yw.x}}{=}
	\bigg|
	\f{\mu_0(\sigma^1=\yel)
	-\mu_0(\varsigma=\yel\whi)}
		{\mu_0(\sigma^1=\yel)-\zeta(\yel\whi)}
	\zeta(\yel\whi)
	-\zeta_0(\yel\whi)
	\bigg|\stackrel{\eqref{e:ZETA.b.x.error}}{=}
	\bigg|
	\bigg(1 + \f{k^{O(1)}\delta}{2^k}
	 \bigg) \zeta(\yel\whi)-\zeta_0(\yel\whi)
	\bigg|\\
	&\stackrel{\eqref{e:ZETA.assume.ww}}{=} \bigg|
	\zeta(\yel\whi)-\zeta_0(\yel\whi)
	+ \f{k^{O(1)}\delta}{4^k}
	\bigg|
	\stackrel{\eqref{e:ZETA.delta.error}}{\le }
	\f{k^{O(1)}\delta}{2^k}\,.
	\label{e:ZETA.x.error.add}
	\end{align}
Next let us consider events that are only affected by the $b$-update: using \eqref{e:ZETA.b.x.error} gives
	\begin{align*}
	\Bigg|\Gamma\Bigg(\BARC
	\sigma=\yel\yel,\\
	\varsigma=\yel\yel
	\EARC\Bigg)\Bigg|
	&=\mu_0\Bigg(\BARC
	\sigma=\yel\yel,\\
	\varsigma=\yel\yel
	\EARC\Bigg)\cdot
	\Bigg|
	\f{(b_{1/4})^1(\yel)}{(b_0)^1(\yel)}-1 
	\Bigg|
	\stackrel{\eqref{e:ZETA.assume.iota}}{\le} 
	\f{k^{O(1)} \hit \delta}{2^k} \,,\\
	\Bigg|\Gamma
	\Bigg(\BARC
	\sigma=\yel\yel,\\
	\varsigma=\whi\yel
	\EARC\Bigg)\Bigg|
	&=\mu_0\Bigg(\BARC
	\sigma=\yel\yel,\\
	\varsigma=\whi\yel
	\EARC\Bigg)\cdot
	\Bigg|
	\f{(b_{1/4})^1(\yel)}{(b_0)^1(\yel)}-1\Bigg|
	\stackrel{\eqref{e:ZETA.assume.ww}}{\le} 
	\f{k^{O(1)} \delta}{4^k}\,.
	\end{align*}
Similarly, for any $\vartheta\in\set{\yel,\whi}\times\set{\red,\yel,\cya}$ we have,
again using \eqref{e:ZETA.b.x.error}, that
	\beq\label{e:ZETA.sigma.one.yellow.varsigma}
	\Bigg|\Gamma\Bigg(\BARC
		\sigma^1=\yel,\\
		\varsigma=\vartheta
		\EARC\Bigg)\Bigg|
	=\mu_0\Bigg(\BARC
		\sigma^1=\yel,\\
		\varsigma=\vartheta
		\EARC\Bigg)
	\cdot\Bigg|
		\f{(b_{1/4})^1(\yel)}{(b_0)^1(\yel)}-1
		\Bigg|
	\stackrel{\eqref{e:ZETA.assume.ww}}{\le} 
	\f{k^{O(1)} \delta}{4^k}\,.
	\eeq
As a consequence, for any $\vartheta\in\set{\red,\yel,\cya,\whi}^2\setminus\set{\yel\whi,\whi\whi}$ we have
	\beq\label{e:ZETA.at.most.b.update}
	\Big|\Gamma(\varsigma=\vartheta)
\Big|
	=\Bigg| \Gamma
		\Bigg(\BARC
		\sigma^1=\yel,\\
		\varsigma=\vartheta
		\EARC\Bigg)
	+\overbrace{\Gamma
		\Bigg(\BARC
		\sigma^1\ne\yel,\\
		\varsigma=\vartheta
		\EARC\Bigg)}^\textup{zero}
		\Bigg|
	\stackrel{\eqref{e:ZETA.sigma.one.yellow.varsigma}}{\le}
	\f{k^{O(1)} \delta}{4^k}\,.
	\eeq
Similarly we also have the bound
	\beq\label{e:ZETA.combined.ww.yw}
	\Bigg|
	\Gamma\Bigg(
		\BARC\sigma^1=\yel,\\
		\varsigma\in\set{\whi\whi,\yel\whi}
		\EARC\Bigg)\Bigg|
	=\Bigg| \overbrace{
	\Gamma(\sigma^1=\yel)
	}^\textup{zero by
		\eqref{e:ZETA.same.normalization}}
	{} - \Gamma\Bigg(
		\BARC\sigma^1=\yel,\\
		\varsigma
		\in\set{\yel,\whi}
		\times\set{\red,\yel,\cya}
		\EARC\Bigg)
	\Bigg|
	\stackrel{\eqref{e:ZETA.sigma.one.yellow.varsigma}}{\le}
	\f{k^{O(1)}\delta}{4^k}\,.
	\eeq
Next we note that under both $\mu_0$ and $\mu_{1/4}$, the conditional probability of $\sigma^2$ given $\sigma^1=\yel$ and $\varsigma$ is given by
	\[
	\mu_t\Bigg(\sigma^2=\yel\,\Bigg|\,
		\BARC\sigma^1=\yel,\\
		\varsigma=\vartheta\EARC\Bigg)
	=\begin{cases}
	1 & \textup{if $\vartheta\in
		\set{\yel\yel,\whi\yel}$,}\\
	\DS\f{\dq(\yel\yel)(b_0)^2(\yel)}
		{\dq(\yel\yel)(b_0)^2(\yel)
			+\dq(\yel\cya)(b_0)^2(\cya)}
	\le O(1)
	&\textup{if $\vartheta\in
		\set{\yel\whi,\whi\whi}$,}
	\end{cases}
	\]
for both $t=0$ and $t=1/4$. It follows that
	\[
	\Big|\Gamma(\sigma=\yel\yel)\Big|
	=\Bigg|
	\underbrace{\Gamma\Bigg(\BARC\sigma^1=\yel,\\
		\varsigma\in\set{\yel\yel,\whi\yel}\EARC\Bigg)
		}_{\textup{bounded by \eqref{e:ZETA.sigma.one.yellow.varsigma}}}
		\cdot1
		+
	\underbrace{
	\Gamma\Bigg(\BARC
	\sigma^1=\yel,\\
	\varsigma\in\set{\whi\whi,\yel\whi}
	\EARC\Bigg)
	}_{\textup{bounded by \eqref{e:ZETA.combined.ww.yw}}}
	\cdot O(1)
		\Bigg|
	\le \f{k^{O(1)}\delta}{4^k}\,.
	\]
Similar arguments (details omitted) can be used to bound $\Gamma(\sigma=\vartheta)$ for all $\vartheta\in\set{\yel,\cya}^2$, so altogether we have
	\beq\label{e:ZETA.change.in.sigma.cysquared}
	\max\Bigg\{
	\Big| \Gamma(\sigma=\vartheta)\Big|
	:\vartheta\in\set{\yel,\cya}^2
	\Bigg\}
	\le \f{k^{O(1)}\delta}{4^k}\,.
	\eeq
Lastly we note that $\Gamma(\sigma^1=\red)=0$, while
	\beq\label{e:ZETA.yw.update.doesnt.change.red.mgl}
	\Big|\Gamma(\sigma^2=\red)\Big|
	=\Bigg|\Gamma\Bigg(\BARC
		\sigma^1=\yel,\\
		\sigma^2=\red\EARC
		\Bigg)
	+\underbrace{
	\Gamma\Bigg(\BARC
		\sigma^1\ne\yel,\\
		\sigma^2=\red\EARC
		\Bigg)
	}_\textup{zero}
	\Bigg|
	\stackrel{\eqref{e:ZETA.at.most.b.update}}{\le}
	\f{k^{O(1)} \delta}{4^k}\,.\eeq
This concludes our analysis of the effect of 
\eqref{e:ZETA.update.yw.b} and \eqref{e:ZETA.update.yw.x} alone on the $\sigma$-marginal.\smallskip

\noindent\bemph{Part 3. Effect of updates \eqref{i:ZETA.update.yw} and \eqref{i:ZETA.update.wy} on $\sigma$-marginal.} Recall that $\mu_{1/2}$ is the measure that results after completing updates \eqref{i:ZETA.update.yw} and \eqref{i:ZETA.update.wy}, and denote the signed measure $\Delta\equiv\mu_{1/2}-\mu_0$. Recall that \eqref{e:ZETA.change.in.sigma.cysquared} bounds $|(\mu_{1/8}-\mu_0)(\sigma=\vartheta)|$ for all $\vartheta\in\set{\yel,\cya}^2$, where from time $t=0$ to time $t=1/8$ we perform only the $\yel\whi$-update.
From time $t=1/8$ to time $t=1/2$ we perform analogously the $\cya\whi$-, $\whi\yel$-, and $\whi\cya$-updates, for which the analogous estimate holds. Therefore
	\[\max\Bigg\{\Big|\Delta(\sigma=\vartheta)\Big|
	:\vartheta\in\set{\yel,\cya}^2\Bigg\}
	\le \f{k^{O(1)}\delta}{4^k}\,.\]
On the other hand, the $\set{\sigma=\red\red}$ event is completely unaffected by updates
\eqref{i:ZETA.update.yw} and \eqref{i:ZETA.update.wy},
so $\Delta(\sigma=\red\red)=0$. It remains to 
estimate $\Delta(\sigma=\vartheta)$ for $\vartheta\in\set{\RYC}^2$ with $\red[\vartheta]=1$. To this end we note that for any $\vartheta\in\set{\RYC}^2$, the event $\set{\varsigma=\vartheta}$ is affected only by the $b$-updates, so
	\beq\label{e:ZETA.bound.add.error.RYCsq}
	\Delta(\varsigma=\vartheta)
	\stackrel{\eqref{e:ZETA.b.x.error}}{=}
	\f{k^{O(1)}\delta}{2^k}
	\zeta_0(\vartheta)
	\stackrel{\eqref{e:ZETA.assume.iota}}{\le}
	\f{k^{O(1)} \hit \delta}{2^k}
	\stackrel{\eqref{e:ZETA.assume.iota}}{\le}
	\f{k^{O(1)} \delta}{2^{k(2+\EPSINT)}}
	\,.\eeq
By combining \eqref{e:ZETA.bound.add.error.RYCsq} with \eqref{e:ZETA.yw.update.doesnt.change.red.mgl} we obtain 
	\beq\label{e:ZETA.varsigma.rw.better.estimate}
	\Big| \Delta(\varsigma=\red\whi)\Big|
	=\bigg|
	\Delta(\sigma^1=\red)
	-\Delta(\varsigma\in
		\set{\red}\times\set{\red,\yel,\cya})
	\bigg|
	\le \bigg(\f1{2^k} + \hit \bigg)
	\f{k^{O(1)}\delta}{2^k} 
	\stackrel{\eqref{e:ZETA.assume.iota}}{\le} 
	\f{k^{O(1)}\delta}{4^k} 
	\,.\eeq
The probability of $\sigma=\red\yel$
conditional on $\varsigma=\red\whi$ is given by
	\beq\label{e:ZETA.ry.given.rw}
	\mu_{1/2}\Big(\sigma=\red\yel
		\,\Big|\,\varsigma=\red\whi\Big)
	=\f{\dq(\red\yel) (b_{1/2})^2(\yel)}
		{\dq(\red\yel) (b_{1/2})^2(\yel)
		+\dq(\red\cya) (b_{1/2})^2(\cya)}
	\stackrel{\eqref{e:ZETA.b.x.error}}{=}
	\mu_0\Big(\sigma=\red\yel\,
		\Big|\,\varsigma=\red\whi
		\Big)
	\Bigg\{1 + \f{k^{O(1)} \delta}{2^k}\Bigg\}\,.
	\eeq
Meanwhile $\mu_{1/2}(\sigma=\red\yel\,|\,\varsigma=\red\yel)=1$. Combining the last few estimates gives
	\begin{align*}
	\mu_{1/2}(\sigma=\red\yel)
	&=\overbrace{
		\Bigg\{ \mu_{1/2}(\varsigma=\red\whi)
		\cdot\mu_{1/2}\Big(\sigma=\red\yel
			\,\Big|\,\varsigma=\red\whi\Big)\Bigg\}
		}^{\textup{estimate by
			\eqref{e:ZETA.varsigma.rw.better.estimate}
			and \eqref{e:ZETA.ry.given.rw}}}
		+
		\overbrace{\Bigg\{
		\mu_{1/2}(\varsigma=\red\yel)\cdot1\Bigg\}
		}^{\textup{estimate by \eqref{e:ZETA.bound.add.error.RYCsq}}}\\
	&=\Bigg\{ \mu_0(\varsigma=\red\whi)
	+ \f{k^{O(1)}\delta}{4^k} 
	\Bigg\}
	\mu_0(\sigma=\red\yel\,|\,\varsigma=\red\whi)
	\Bigg\{1 + \f{k^{O(1)} \delta}{2^k}\Bigg\}
	+\Bigg\{
	\mu_0(\varsigma=\red\yel)+
		\f{k^{O(1)}\hit\delta}{2^k}\Bigg\}\\
	&\stackrel{\eqref{e:ZETA.assume.ww}}{=}
	\mu_0(\sigma=\red\yel)
	+ \bigg(\f1{2^k} + \hit \bigg)
	\f{k^{O(1)}\delta}{2^k} 
	\stackrel{\eqref{e:ZETA.assume.iota}}{=} 
	\mu_0(\sigma=\red\yel)+
	\f{k^{O(1)}\delta}{4^k}\,.
	\end{align*}
The analogous estimate holds for the event $\set{\sigma=\vartheta}$ for all $\vartheta\in\set{\RYC}^2$ with $\red[\vartheta]=1$. Altogether, if we write $\omega_t$ for the marginal law of $\sigma$ under $\mu_t$, then we have
	\beq\label{e:ZETA.crelerr.a.b}
	\crelerr(\omega_0,\omega_{1/2})
	\le k^{O(1)} \begin{pmatrix}
	4^{-k} \\ 2^{-k} \\ 0
	\end{pmatrix} \delta\,,\eeq
having made use of \eqref{e:ZETA.mu.zero.estimates}.
\smallskip

\noindent\bemph{Part 4. Effect of updates \eqref{i:ZETA.update.yw} and \eqref{i:ZETA.update.wy} on $\varsigma$-marginal.} Recall that $\zeta_t$ denotes the marginal law of $\varsigma$ under $\mu_t$. 
Update \eqref{i:ZETA.update.yw} results in 
$\zeta_{1/4}(\varsigma)=\zeta(\varsigma)$
for $\varsigma\in\set{\yel\whi,\cya\whi}$. Update \eqref{i:ZETA.update.wy} 
results in $\zeta_{1/2}(\varsigma)=\zeta(\varsigma)$
for $\varsigma\in\set{\whi\yel,\whi\cya}$, but it need not hold that
$\zeta_{1/2}(\varsigma)=\zeta(\varsigma)$
for $\varsigma\in\set{\yel\whi,\cya\whi}$.
However we claim that the discrepancy is very small. 
Recall that as a consequence of \eqref{e:ZETA.same.normalization} we have $z_0=z_{1/4}=z_{1/2}$. Therefore
	\begin{align*}
	\mu_{1/2}(\varsigma=\yel\whi)
	&\stackrel{\eqref{e:ZETA.mu.t}}{=}
	\f1{z_{1/2}}
	x_{1/2}(\yel\whi)
	(b_{1/2})^1(\yel)
	\Bigg\{ \dq(\yel\yel)(b_{1/2})^2(\yel)
	+\dq(\yel\cya)(b_{1/2})^2(\cya)
	\Bigg\} \\
	&\stackrel{\eqref{e:ZETA.b.x.error}}{=}
	\f1{z_{1/4}}
	x_{1/4}(\yel\whi)
	(b_{1/4})^1(\yel)
	\Bigg\{\dq(\yel\yel)(b_{1/4})^2(\yel)
	+\dq(\yel\cya)(b_{1/4})^2(\cya)
	\Bigg\}
	\Bigg\{1 + \f{k^{O(1)}\delta}{2^k}\Bigg\}\\
	&\stackrel{\eqref{e:ZETA.mu.t}}{=}
	\mu_{1/4}(\varsigma=\yel\whi)
	\Bigg\{1 + \f{k^{O(1)}\delta}{2^k}\Bigg\}
	\stackrel{\eqref{e:ZETA.assume.ww}}{=} 
	\zeta(\yel\whi)
	+ \f{k^{O(1)}\delta}{4^k}\,.
	\end{align*}
The analogous estimate holds for $\varsigma=\cya\whi$.
Denote the signed measure $\Upsilon\equiv\zeta_{1/2}-\zeta$; the above can be rewritten as
	\beq\label{e:ZETA.yw.error.after.four.updates}
	\max\Bigg\{
	\Big|\Upsilon(\yel\whi) \Big|,
	\Big|\Upsilon(\cya\whi) \Big|\Bigg\}
	=
	\Big|(\zeta_{1/2}-\zeta)(\yel\whi) \Big|
	\le \f{k^{O(1)}\delta}{4^k}\,.
	\eeq
Next, for all $\vartheta\in\set{\RYC}^2$ we have
	\beq\label{e:ZETA.rycsq.error.after.four.updates}
	\Big|\Upsilon(\vartheta)\Big|
	=\bigg| (\zeta_{1/2}-\zeta_0)(\vartheta)
	+(\zeta_0-\zeta)(\vartheta)\bigg|
	\stackrel{\eqref{e:ZETA.bound.add.error.RYCsq}}
		{\le}
	\f{k^{O(1)} \hit \delta}{2^k}
	+\Big| (\zeta_0-\zeta)(\vartheta)\Big|
	\stackrel{\eqref{e:ZETA.delta.error.hat.iota}}
		{\le}
		O(\hit\delta)\,.\eeq
Recall that $\mu_0(\sigma^i=\red)=\starpi_e(\red)=\zeta^i(\red)$
for $i=1,2$. Therefore, using
\eqref{e:ZETA.varsigma.rw.better.estimate} 
and \eqref{e:ZETA.rycsq.error.after.four.updates}, we have
	\beq\label{e:ZETA.varsigma.rw.zeta.error}
	\Upsilon(\red\whi)
	=\Bigg|(\mu_{1/2}-\mu_0)(\sigma^1=\red)
	-\sum_{\vartheta\in\set{\red}\times\set{\RYC}}
	\Upsilon(\vartheta)\Bigg|
	\le\bigg( \f{1}{4^k}
		+ \hit \bigg) k^{O(1)} \delta
	\stackrel{\eqref{e:ZETA.assume.iota}}{\le} 
	k^{O(1)}\hit\delta \,.\eeq
Combining
\eqref{e:ZETA.yw.error.after.four.updates},
\eqref{e:ZETA.rycsq.error.after.four.updates}, and \eqref{e:ZETA.varsigma.rw.zeta.error} gives
	\beq\label{e:ZETA.varsigma.ww.zeta.error}
	\Big|\Upsilon(\whi\whi)\Big|
	= \Bigg|-\Upsilon
		\Bigg(\BARC(\set{\whi}\times\set{\RYC})\\
		\cup(	\set{\RYC}\times\set{\whi})
		\EARC
		\Bigg)
	-\Upsilon\Big(\set{\red,\yel,\cya}^2\Big)
	\Bigg|
	\le
	\bigg(\f1{4^k} + \hit \bigg)
		k^{O(1)}\delta
	\stackrel{\eqref{e:ZETA.assume.iota}}{\le} 
	k^{O(1)}\hit\delta\,.
	\eeq
This concludes our analysis of updates \eqref{i:ZETA.update.yw} and \eqref{i:ZETA.update.wy}.\smallskip

\noindent\bemph{Part 5. Effect of update \eqref{i:ZETA.update.x.only}.} It is clear that we will have $\zeta_{3/4}=\zeta$, so it remains to understand the effect of update \eqref{i:ZETA.update.x.only} on the $\sigma$-marginal. 
Note that since the $b$-weights remain the same during update \eqref{i:ZETA.update.x.only}, the conditional probabilities of $\sigma$ given $\varsigma$ remain unchanged between times $t=1/2$ and $t=3/4$. Thus we have
	\[
	\mu_{3/4}(\sigma=\yel\yel)
	= \sum_{\vartheta\in\set{\yel,\whi}^2}
	\mu_{3/4}(\varsigma=\vartheta)
	\mu_{1/2}(\sigma=\yel\yel\,|\,\varsigma=\vartheta)
	= \mu_{1/2}(\sigma=\yel\yel)
	+ k^{O(1)} \hit\delta\,,
	\]
where the last estimate uses
\eqref{e:ZETA.yw.error.after.four.updates}, 
\eqref{e:ZETA.rycsq.error.after.four.updates},
and \eqref{e:ZETA.varsigma.ww.zeta.error}. The analogous bound holds for $\set{\sigma=\vartheta}$ for each $\vartheta\in\set{\yel,\cya}^2$. Similarly we also have
	\[\mu_{3/4}(\sigma=\red\yel)
	=\Bigg\{ \mu_{3/4}(\varsigma=\red\whi)
	\mu_{1/2}\Big(\sigma=\red\yel\,\Big|\,\varsigma=\red\whi\Big)\Bigg\}
	+ \mu_{3/4}(\varsigma=\red\yel)
	\stackrel{\eqref{e:ZETA.varsigma.rw.zeta.error}}{=}
	\mu_{3/4}(\sigma=\red\yel)
	+ k^{O(1)}\hit\delta\,,\]
and the analogous bound holds for all $\sigma\in\set{\RYC}^2$ with $\red[\sigma]=1$.
Lastly, for the case $\sigma=\red\red$, we note that
	\[
	\f{\mu_{3/4}(\sigma=\red\red)}
		{\mu_{1/2}(\sigma=\red\red)}
	= \f{\mu_{3/4}(\varsigma=\red\red)}
		{\mu_0(\varsigma=\red\red)}
	=\f{\zeta(\red\red)}
		{\zeta_0(\red\red)}
	\stackrel{\eqref{e:ZETA.delta.error.mult.rr}}{=}
	 1 + O(\ddot{\delta})\,.
	\]
Combining the above estimates with 
\eqref{e:ZETA.mu.zero.estimates} and 
\eqref{e:ZETA.crelerr.a.b} gives
	\[
	\crelerr(\omega_0,\omega_{3/4})
	\le
	k^{O(1)} \begin{pmatrix}
	\hit & 0 \\ 2^k\hit&0 \\ 1&1
	\end{pmatrix}
	\begin{pmatrix} \delta \\ \ddot{\delta}
		\end{pmatrix} \,.
	\]
This can be translated to a similar error bound for the clause-to-variable messages via the relation
	\[
	\hq_t(\sigma)
	=\f{ \omega_t(\sigma)/\dq(\sigma)}
		{\sum_{\sigma'}
		\omega_t(\sigma')/\dq(\sigma')}\,.
	\]
The claimed result follows.
\end{proof}
\end{lem}

\begin{proof}[\hypertarget{pf:p.balancedeVertex}{Proof of Proposition~\ref{p:balancedeVertex}}] We shall apply the result of Lemma~\ref{l:ZETA.conditions.balanced.vertex}. Condition~\eqref{e:ZETA.ww.unif.bound} of the lemma is satisfied due to Lemma~\ref{l:lotsOfAs}. Condition~\eqref{e:ZETA.hat.iota.unif.bound} of the lemma is satisfied due to Lemma~\ref{l:lotsOfAsDiverse}, using the assumption that $\notDiverse(v)=\notLight(v)=0$. The claimed result immediately follows from the bounds \eqref{e:balanceMarg} of Lemma~\ref{l:ZETA.conditions.balanced.vertex}.\end{proof}

\subsection{Non-defective variables neighboring mostly diverse light clauses}
\label{ss:proof.p.DIVERSE.VAR}

In \S\ref{ss:proof.propn.balancedeVertex} we proved Proposition~\ref{p:balancedeVertex}
under the assumption that $v$ is a non-defective variable with $\notDiverse(v)=\notLight(v)=0$. 
In the current subsection we state and prove Proposition~\ref{p:DIVERSE.VAR}, which holds under the weaker assumption that $\notDiverse(v)$ and $\notLight(v)$ are small but not necessarily zero. (The estimates given by Proposition~\ref{p:DIVERSE.VAR} are also weaker than those given by Proposition~\ref{p:balancedeVertex}.)

\begin{ppn}[used only in proof of Proposition~\ref{p:noNonDiverse}]\label{p:DIVERSE.VAR}
For any $\EPSP>0$ there exists $k_0<\infty$ large enough (depending only on $\EPSP$) such that the following holds for all $k\ge k_0$. Suppose $v$ is a non-defective variable such that (in the notation of Lemma~\ref{l:large11Case})
	\beq\label{e:assumption.notD.notL}
	\max\Big\{\notDiverse(v),
		\notLight(v)\Big\}
		\le \f{2^k}{2^{k\EPSP}}\,.
	\eeq
Then $v$ is diverse (Definition~\ref{d:diverse}), and furthermore $\pi_v$ satisfies the estimate
	\beq\label{eq:diverseVariableCondition}
	\max\Bigg\{\bigg|\pi_v(x)-\f14\bigg|
	: x\in\set{\minus,\plus}^2
	\Bigg\}\le\f1{100}\,.\eeq

\begin{proof} Without loss of generality we can assume that $\EPSP$ is much smaller than the other constants 
$\EPSONE,\EPSTWO,\EPSTHREE,\EPSDIV,\EPSLIGHT$ appearing in this section.

We will prove the result by contradiction, so let $v$ be a variable of type $\bT$ which satisfies the conditions of the proposition, but for which the estimate~\eqref{eq:diverseVariableCondition} fails to hold. It will be used repeatedly in the proof that since $v$ is non-defective (Definition~\ref{d:j.defective}), it must be nice, and all clause types neighboring to $v$ must also be nice (Definition~\ref{d:nice}). As in \eqref{e:uX.notation} and \eqref{e:cal.X} we let $X_e\equiv(\sigma_e,\varsigma_e)$. We also denote $\cX_e\equiv(X_e,\bL_e)$. We consider the law $\bP$ of the random variable 
	\[
	\cX \equiv \cX_v \equiv (\cX_e)_{e\in\delta v}
	\equiv (\uX_{\delta v},\uL_{\delta v})
	\equiv\bigg((X_e)_{e\in\delta v}, 
	(\bL_e)_{e\in\delta v}\bigg)\,.\]
The measure $\bP$ must maximize entropy subject to the following constraints (cf.\ \eqref{e:final.D.L.clause}--\eqref{e:final.D.L.varsigma}):
for all $e\in\delta v$, 
	{\setlength{\jot}{0pt}\begin{alignat}{2}
	\bP(\bL_e=\bL)
	&=\pi_{\DD}(\bL\,|\,\bt_e)
	\quad&&\textup{for all $\bL$,}
	\label{e:D.L.diverse.clause} \\ 
	\bP(\sigma^i_e=\sigma^i \,|\,
		\bL_e=\bL)
	&= \starpi_e(\sigma^i)
	\quad&&\textup{for all $\bL$ and
		 all $\sigma^i\in\set{\RYGB}$,}
	\label{e:D.L.diverse.judicious} \\
	\bP(\varsigma_e=\varsigma
		\,|\,\bL_e=\bL)
	&=\bzeta_e(\varsigma\,|\,\bL)
	\quad&&\textup{for all $\bL$
		and all $\varsigma\in\set{\RYC,\whi}^2$.}
	\label{e:D.L.diverse.varsigma}
	\end{alignat}}%
Via a series of transformations of the measure $\bP$ we will construct another distribution $\bPe$ satisfying the same constraints
\eqref{e:D.L.diverse.clause}--\eqref{e:D.L.diverse.varsigma} but with higher entropy, yielding the required contradiction. The remainder of the proof is outlined as follows:
\begin{enumerate}[--]
\item In \hyperlink{l:DIVERSE.product.Pa}{Part 1}
we define a measure $\bPa$ satisfying
constraints \eqref{e:D.L.diverse.clause}, \eqref{e:D.L.diverse.judicious}, and 
\eqref{eq:relaxedZetaConstraint}, where
\eqref{eq:relaxedZetaConstraint} is a relaxation of
\eqref{e:D.L.diverse.varsigma}. We show that $\bPa$ has a simple explicit form, and has substantially larger entropy than $\bP$ (see \eqref{e:rel.ent.diverse.absolute.lbd} below).

\item In \hyperlink{l:DIDERSE.Pa.bounds}{Part 2} we prove probabilistic estimates on various quantities under the measure $\bPa$. Let $\cXa$ denote a sample from $\bPa$; in the rest of the proof we transform $\cXa$ without changing its frozen spin $x=x^a\in\set{\minus,\plus,\free}^2$.

\item In \hyperlink{l:DIVERSE.Pb}{Part 3} we transform $\cXa$ into $\cXb$, and in \hyperlink{l:DIVERSE.Pc}{Part 4} we transform $\cXb$ into $\cXc$. In these transformations, $(\varsigma_e)^i$ can only change between \SPIN{red} and \SPIN{white}, so any $(\varsigma_e)^i\in\set{\yel,\cya}$ remains unchanged. The transformations combined will guarantee (see \eqref{e:DIVERSE.hat.E.nonempty} below) that in the configuration $\cXc$, if $x=x^c$ has $x^i=z\in\set{\minus,\plus}$, then some $e\in\delta v(z)$ will have $(\sigma^{c,i})_e=\red$, $(\sigma^c)_e\ne\red\red$, and $(\bL^c)_e\in\mathbb{D}\cap\mathbb{L}$. The condition \eqref{e:DIVERSE.hat.E.nonempty} will allow more flexibility to change spins from non-\SPIN{white} to \SPIN{white} in the following step. The law $\bPc$ of $\cXc$ satisfies
\eqref{e:D.L.diverse.clause}, but need not satisfy
\eqref{e:D.L.diverse.judicious} or \eqref{e:D.L.diverse.varsigma}, in particular, it can have less than the correct density of \SPIN{red} spins (see \eqref{eq:Xc.red.marginal.div}
and \eqref{eq:Xc.red.marginal.div2}).

\item In \hyperlink{l:DIVERSE.Pd}{Part 5} we transform $\cXc$ to $\cXd$ by changing $(\varsigma^{c,i})_e\ne\whi$ to $(\varsigma^{d,i})_e=\whi$ in some cases --- the condition \eqref{e:DIVERSE.hat.E.nonempty} from the previous step allows more flexibility to do this without invalidating the configuration. The goal of this step is to create enough \SPIN{white} spins which can be altered in the next step to ensure that the constraints \eqref{e:D.L.diverse.judicious} or \eqref{e:D.L.diverse.varsigma} will be satisfied.
At the same time, we cannot change too many spins to \SPIN{white}, since this could result in too much decrease in entropy.

\item In \hyperlink{l:DIVERSE.Pe}{Part 6} we transform
$\cXd$ to $\cXe$, where we use $\cXe$ rather than $\cX^e$ to avoid confusion with edge labels $e\in\delta v$. In this final step we change
 $(\varsigma^{d,i})_e=\whi$
to $(\varsigma^{f,i})_e\ne\whi$ in some cases, such that the law $\bPe$ of $\cXe$ satisfies the original constraints
\eqref{e:D.L.diverse.clause}--\eqref{e:D.L.diverse.varsigma}
(see \eqref{e:DIVERSE.final.msr.sat.varsigma} below).
Finally we show that $\bPe$ does not have much smaller entropy than $\bPa$, and as a result must have substantially larger entropy than $\bP$, giving the contradiction.
\end{enumerate}
We now turn to the details of the transformation.\smallskip

\noindent\bemph{\hypertarget{l:DIVERSE.product.Pa}{Part 1}. Product solution $\bPa$ for relaxed constraints.} Let $\bPa$ be the probability measure over configurations $\cX$ which maximizes entropy subject to the constraints \eqref{e:D.L.diverse.clause}, \eqref{e:D.L.diverse.judicious}, and (in place of \eqref{e:D.L.diverse.varsigma})
	\beq\label{eq:relaxedZetaConstraint}
	\bPa(\varsigma_e^i=\varsigma^i
			\,|\,\bL_e=\bL)
		=\bzeta_e(\varsigma^i\,|\,\bL)
		\quad\textup{for all $\bL$
			and all
			$\varsigma^i\in\set{\RYC,\whi}$.}
	\eeq
Let us note that the law $\bPa$ has a quite simple form: by the method of Lagrange multipliers, it must be expressible as
(cf.\ \eqref{e:final.D.L.Lagrangian.form})
	\beq\label{e:DIVERSE.P.a.Lagrangian.form}
	\bPa\Big(
	\usi_{\delta v},\uvsi_{\delta v},
	\uL_{\delta v} \Big)
	\cong \psi_v(\usi_{\delta v})
	\prod_{e\in\delta v}
	\Bigg\{
	\Ind{\sigma_e\sim\varsigma_e}
	\psi_e(\bL_e)
	\prod_{i=1,2}
	\bigg[
	(\beta_e)^i((\sigma_e)^i\,|\,\bL_e)
	(\chi_e)^i( (\varsigma_e)^i\,|\,\bL_e)
	\bigg]
	\Bigg\}\,,
	\eeq
where $\psi_v(\usi_{\delta v})$ is parametrized as in
\eqref{e:non.compound.weight.functional.form} from Definition~\ref{d:param.weights.Augmented}. A valid solution is given by simply setting $\psi_v\equiv1$,
$\psi_e(\bL)=\pi_{\DD}(\bL\,|\,\bt_e)$,
and (cf.\ \eqref{e:ZETA.bstar} and \eqref{e:ZETA.xstar})
	\begin{align*}
	\begin{pmatrix}
	(\beta_e)^i(\red\,|\,\bL)\\
	(\beta_e)^i(\yel\,|\,\bL)\\
	(\beta_e)^i(\grn\,|\,\bL)
	=(\beta_e)^i(\blu\,|\,\bL)\equiv(\beta_e)^i(\cya\,|\,\bL)
	\end{pmatrix}
	&= 
	\begin{pmatrix}
	1\\
	[\starpi_e(\yel)-(\bzeta_e)^i(\yel\,|\,\bL)]/\dqstar_e(\yel)\\
	[\starpi_e(\cya)-(\bzeta_e)^i(\cya\,|\,\bL)]/\dqstar_e(\cya)
	\end{pmatrix}
	\stackrel{\eqref{e:ZETA.zetastar}}{=} 
	\begin{pmatrix}
	\Theta(1)\\ \Theta(1) \\ \Theta(1)
	\end{pmatrix}\,,\\
	\begin{pmatrix}
	(\chi_e)^i(\red\,|\,\bL)\\
	(\chi_e)^i(\yel\,|\,\bL)\\
	(\chi_e)^i(\cya\,|\,\bL)\\
	(\chi_e)^i(\whi\,|\,\bL)
	\end{pmatrix}
	&=\begin{pmatrix}
	\starpi_e(\red)/\dqstar_e(\red)\\
	(\bzeta_e)^i(\yel\,|\,\bL)
		/[\starpi_e(\yel)-(\bzeta_e)^i(\yel\,|\,\bL)]\\
	(\bzeta_e)^i(\cya\,|\,\bL)
	/[\starpi_e(\cya)
		-(\bzeta_e)^i(\cya\,|\,\bL)]\\
	1
	\end{pmatrix}
	\stackrel{\eqref{e:ZETA.zetastar}}{=} 
	\begin{pmatrix}
	\Theta(1/2^k)\\
	O(k^2/2^k)\\
	O(k^2/2^k)\\
	1
	\end{pmatrix}
	\,,\end{align*}
where the estimates on
$(\bzeta_e)^i(\tau\,|\,\bL)$
for $\tau\in\set{\RYC}$ come from Lemma~\ref{l:lotsOfAs}. Note that for $e\in\delta v$, $\bL=\bL_e$, $i=1,2$, and $\sigma^i\in\set{\RYGB}$ we have
	\beq\label{e:product.zeta.msg.simplification}
	(\hq_e)^i(\sigma^i\,|\,\bL)\equiv
	(\beta_e)^i(\sigma^i\,|\,\bL)
	\sum_{\varsigma^i}
	\Ind{\varsigma^i \sim \sigma^i}
	(\chi_e)^i(\varsigma^i\,|\,\bL)
	= \f{\starpi_e(\sigma^i)}{\dqstar_e(\sigma^i)}\,,
	\eeq
where we emphasize that the right-hand side of \eqref{e:product.zeta.msg.simplification} does not depend on $\bL$. Indeed,
	\[
	(\hq_e)^i(\red\,|\,\bL)
	= (\beta_e)^i(\red\,|\,\bL)
	(\chi_e)^i(\red\,|\,\bL)
	= \f{\starpi_e(\red)}{\dqstar_e(\red)}\,,
	\]
which verifies \eqref{e:product.zeta.msg.simplification} for the case $\sigma^i=\red$. Next,
	\begin{align*}
	(\hq_e)^i(\yel\,|\,\bL)
	&= (\beta_e)^i(\yel\,|\,\bL)
	\bigg(
	(\chi_e)^i(\yel\,|\,\bL)
	+(\chi_e)^i(\whi\,|\,\bL)\bigg)\\
	&= \f{\starpi_e(\yel)-(\bzeta_e)^i(\yel\,|\,\bL)}
		{\dqstar_e(\yel)}
	\cdot\Bigg\{
	\f{(\bzeta_e)^i(\yel\,|\,\bL)}
		{\starpi_e(\yel)-(\bzeta_e)^i(\yel\,|\,\bL)}+1
	\Bigg\}
	= \f{\starpi_e(\yel)}{\dqstar_e(\yel)}\,,
	\end{align*}
which verifies \eqref{e:product.zeta.msg.simplification} for the case $\sigma^i=\yel$. A similar calculation gives
\eqref{e:product.zeta.msg.simplification} for $\sigma^i\in\set{\grn,\blu}$. It follows from \eqref{e:product.zeta.msg.simplification},
together with the identity $\starpi_e(\tau)/\dqstar_e(\tau) \cong \hqstar_e(\tau)$,
that the marginal of \eqref{e:DIVERSE.P.a.Lagrangian.form} over
$(\usi_{\delta v},\uL_{\delta v})$ is given by
	\begin{align}\nonumber
	\bPa\Big(\usi_{\delta v},\uL_{\delta v}\Big)
	&\cong
	\varphi_v(\usi_{\delta v})
	\prod_{e\in\delta v}\Bigg\{
	\pi_{\DD}(\bL_e\,|\,\bt_e)
	\prod_{i=1,2}
	\hqstar_e((\sigma_e)^i)\Bigg\} \\
	& \cong
	\dbhstar_v((\usi_{\delta v})^1)
	\dbhstar_v((\usi_{\delta v})^2)
	\prod_{e\in\delta v}
	\pi_{\DD}(\bL_e\,|\,\bt_e)
	\,,\label{eq:productForm}
	\end{align}
where $\dbhstar_v$ is as in 
\eqref{e:opt.variable.tuple.measure.star}. In particular, we see that under $\bPa$,
the random variables $(\usi_{\delta v})^1$, $(\usi_{\delta v})^2$, and $\uL_e$ (for $e\in\delta v$) are mutually independent. Moreover, conditional on $(\usi,\uL)_{\delta v}$, 
the $(\varsigma_e)^i$ are all independent from one another, with conditional laws depending only on
 $(\bL_e,(\sigma_e)^i)$. If $(\sigma_e)^i=\red$ then $(\varsigma_e)^i= \red $ also. 
If $(\sigma_e)^i=\yel$ then $(\varsigma_e)^i\in\set{\yel,\whi}$; and the conditional probability of 
$(\varsigma_e)^i=\yel$ is given by
	\beq\label{eq:ColouringProbs.yel}
	\bPa
	\bigg( (\varsigma_e)^i= \yel
	\,\bigg|\,
	\bL_e=\bL, (\sigma_e)^i=\yel
	\bigg)
	= \f{(\chi_e)^i(\yel\,|\,\bL)}
	{(\chi_e)^i(\yel\,|\,\bL)+(\chi_e)^i(\whi\,|\,\bL)}
	= \f{(\bzeta_e)^i(\yel\,|\,\bL)}
		{\starpi_e(\yel)}
	\le O\bigg(\f{k^2}{2^k}\bigg)\,,
	\eeq
where the last inequality follows by
Lemma~\ref{l:lotsOfAs}. Similarly, if $(\sigma_e)^i =\tau\in\set{\grn,\blu}$ then
$(\varsigma_e)^i\in\set{\cya,\whi}$; and the conditional probability of $(\varsigma_e)^i=\cya$ is given by
	\beq\label{eq:ColouringProbs.cya}
	\bPa
	\bigg( (\varsigma_e)^i=\cya
	\,\bigg|\,
	\bL_e=\bL, (\sigma_e)^i=\tau
	\bigg)
	= \f{(\bzeta_e)^i(\cya\,|\,\bL)}
		{\starpi_e(\blu)+\starpi_e(\grn)}
	\le O\bigg(\f{k^2}{2^k}\bigg)
	\eeq
for $\tau\in\set{\grn,\blu}$. The independence between the two copies $i=1,2$ gives furthermore
	\beq\label{e:DIVERSE.Pa.RYCsq}
	\bPa\bigg(\varsigma_e\in\set{\RYC}^2
		\,\bigg|\, \bL_e=\bL\bigg)
	= \prod_{i=1}^2 \bPa\bigg(
		(\varsigma_e)^i\in\set{\RYC}
		\,\bigg|\, \bL_e=\bL\bigg) 
	\le O\bigg( \f{k^4}{4^k}\bigg)\,.
	\eeq
Note also that if $x=x_v\in\set{\minus,\plus,\free}^2$ is the frozen spin, then $x^1$ and $x^2$ are independent under $\bPa$, and the marginal law of each $x^i$ is determined by the $\starpi_e$ for $e\in\delta v$.
Since we assumed that $v$ is non-defective, it follows that
	\beq\label{eq:diverseVariableCondition.a}
	\max\Bigg\{
	\bigg|\bPa\Big( x =z \Big)
		 -\f14\bigg|
		: z\in\set{\minus,\plus}^2
		\Bigg\} =o_k(1)\,.
	\eeq
If $\bP'$ is any other probability measure satisfying the same constraints as $\bPa$ (namely \eqref{e:D.L.diverse.clause}, \eqref{e:D.L.diverse.judicious}, and \eqref{eq:relaxedZetaConstraint}), then
it follows by the same derivation as
for \eqref{e:relative.entropy.with.Lagrangian} that
the relative entropy between $\bP'$ and $\bPa$ 
satisfies
	\beq\label{e:relent.a.b.transform}
	\Ent(\bP'\,|\,\bPa)
	= \bigg\langle \bP',\log \f{\bP'}{\bPa}
		\bigg\rangle
	=\Ent(\bPa) -\Ent(\bP')\,.
	\eeq
Now suppose that $x\in\set{\minus,\plus,\free}^2$ is the frozen spin corresponding to $\cX\sim\bP$ (meaning that its law fails condition \eqref{eq:diverseVariableCondition}), while $x^a\equiv(x^{a,1},x^{a,2})\in\set{\minus,\plus,\free}^2$ is the frozen spin corresponding to $\cXa\sim\bPa$ (so that its law satisfies condition 
\eqref{eq:diverseVariableCondition.a}). Therefore we obtain
	\beq\label{e:rel.ent.diverse.absolute.lbd}
	\Ent(\bP\,|\,\bPa) \ge
	\DKL(x\,|\,x^a)
	\ge \f1{10^6}\,.
	\eeq
It follows by combining the last few calculations that
	\begin{equation}\label{eq:a.opt.EntComp}
	\Ent(\bPa)
	\stackrel{\eqref{e:relent.a.b.transform}}{=}
	\Ent(\bP)
	+\Ent(\bP\,|\,\bPa)
	\stackrel{\eqref{e:rel.ent.diverse.absolute.lbd}}{\ge}
	\Ent(\bP)
	+\f1{10^6}\,.
	\end{equation}
In the remainder of the proof we will make a sequence of transformations, from $\bPa$ to $\bPb$ and so on, until we arrive at a measure
$\bPe$ that satisfies the original constraints \eqref{e:D.L.diverse.clause}--\eqref{e:D.L.diverse.varsigma}. We will show that
none of the transformations substantially reduce the entropy. Thus the final measure $\bPe$
will have similar entropy as $\bPa$, and hence will have greater entropy than $\bP$ by \eqref{e:rel.ent.diverse.absolute.lbd}. But we assumed $\bP$ to be the measure of maximal entropy satisfying constraints \eqref{e:D.L.diverse.clause}--\eqref{e:D.L.diverse.varsigma}, so this will give the desired contradiction. The measures $\bPa$ through $\bPe$ will all be coupled together, and the frozen spin $x$ will be left unchanged throughout the transformations.\smallskip

\noindent\bemph{\hypertarget{l:DIDERSE.Pa.bounds}{Part 2}. Bounds on $\bPa$.} Using the assumed bound \eqref{e:assumption.notD.notL} on $\notDiverse(v)$ and $\notLight(v)$, we can select a subset of four distinct edges incident to $v$,
	\beq\label{e:four.distinct.edges}
	\bm{E}\equiv
	\bigg\{e^{\plus,1},e^{\plus,2},
	e^{\minus,1},e^{\minus,2}\bigg\}
	\subset \delta v\,,\eeq
where $e^{\plus,i}\in\delta v(\plus)$,
and $e^{\minus,i}\in\delta v(\minus)$,
such that each $e\in\bm{E}$ satisfies
	\beq\label{eq:bEdefin}
	\sum_{\bL}
		\pi_{\DD}(\bL\,|\,\bt_e) 
		\Ind{\bL\in \mathbb{D}\cap\mathbb{L}}
		\ge\f12\,.\eeq
For a configuration $\cX$ as in \eqref{e:cal.X}, we define the following quantities:
	\begin{align*}
\mathfrak{R}^i&\equiv
	\mathfrak{R}^i(\cX)
	\equiv \sum_{e\in\delta v\setminus \bm{E}}
	\mathbf{1}\Big\{
	(\sigma_e)^i=\red
	\Big\}\,,\\
\mathfrak{H}
	&\equiv\mathfrak{H}(\cX)
	\equiv\sum_{e\in\delta v\setminus \bm{E}} 
	\sum_{i=1,2}
	\mathbf{1}\Big\{
	(\varsigma_e)^i\in\{\RYC\}
	\Big\}\,,\\
	\mathfrak{W}
	&\equiv \mathfrak{W}(\cX)
	\equiv\sum_{e\in\delta v\setminus \bm{E}} 
		\mathbf{1}\Big\{
		\varsigma_e=\whi\whi,
		\bL_e \in \mathbb{D}\cap\mathbb{L}
		\Big\}\,,\\
	T&\equiv T(\cX)
		\equiv\mathbf{1}
		\bigg\{\mathfrak{W}(\cX)
			\le \bigg(1-\f1{k^{10}}\bigg)
				|\delta v|\bigg\}\,.
	\end{align*}
We now prove some straightforward probabilistic bounds (under the measure $\bPa$) for the above quantities. For $e\in\delta v$ let $R_e$ denote independent Bernoulli random variables with
	\beq\label{e:DIVERSE.Re}
	\P(R_e=1) 
	= \f{\hqstar_e(\red)}{\hqstar_e(\set{\red,\blu})}
	= \Theta\bigg(\f1{2^k}\bigg)\,.\eeq
We see from \eqref{eq:productForm} that under $\bPa(\cdot\,|\,x^i=\plus)$, the total number of edges $e\in\delta v$ with $(\sigma_e)^i=\red$ is equidistributed as the sum of $R_e$ over $e\in\delta v(\plus)$, conditioned on that sum being strictly positive. Therefore
	\begin{align}\nonumber
	\bPa\Bigg(
		\sum_{i=1,2}\mathfrak{R}^i\ge
		\f{k^3}{2} \Bigg)
	&\le \bPa \Bigg(
	\sum_{i=1,2}
	\sum_{e\in\delta v} \Ind{(\sigma_e)^i=\red}
	\ge \f{k^3}{2} \Bigg)
	\le O(1)\max_{z\in\set{\minus,\plus}}
	\Bigg\{
	\f{\P( \sum_{e\in\delta v(z)} R_e \ge k^3/4)}
		{\P(\sum_{e\in\delta v(z)} R_e \ge1)}
		\Bigg\}\\
	&\le\f{\exp(-\Omega(k^3\log k))
	}{1-o_k(1)}
	\le \f1{\exp(\Omega(k^3\log k))}\,,
	\label{eq:frakR.bound}
	\end{align}
where the second-to-last inequality follows by a Chernoff bound. Next recall from \eqref{eq:ColouringProbs.yel} and 
\eqref{eq:ColouringProbs.cya} that
under $\bPa$, conditional on 
$(\sigma_e)^i\in\set{\yel,\grn,\blu}$, the chance for $(\varsigma_e)^i\in\set{\yel,\cya}$
is $O(k^2/2^k)$. Thus, under $\bPa$,
the total number of edges $e$ with $(\varsigma_e)^i\in\set{\yel,\cya}$ for either $i=1,2$ is stochastically dominated by a binomial random variable with $\Theta(k2^k)$ trials
and success probability $\Theta(1/2^k)$. It follows by another Chernoff bound, and by combining with \eqref{eq:frakR.bound}, that
	\beq\label{e:frakH.bound}
	\bPa\Big(\mathfrak{H} \ge k^3\Big)
	\le \bPa\Bigg(
		\sum_{e\in\delta v}\sum_{i=1,2}
		\mathbf{1}\Big\{
		(\varsigma_e)^i\in\set{\RYC}
		\Big\} \ge k^3
		\Bigg)
	\le\f1{\exp(\Omega(k^3\log k))}\,.\eeq
Next recall from \eqref{eq:productForm}
that the $\bL_e$ for $e\in\delta v$ are mutually independent under $\bPa$. The assumption \eqref{e:assumption.notD.notL} bounds the expected number (under $\bPa$) of edges $e\in\delta v$ with $\bL_e\notin\mathbb{D}\cap\mathbb{L}$. It follows by Azuma's inequality that
	\beq\label{eq:nonDiverseLightBound}
	\bPa
	\Bigg(\sum_{e\in\delta v}
	\Ind{\bL_e\notin\mathbb{D}\cap\mathbb{L}}
	\ge \f{k2^k}{2^{k\EPSP}} \Bigg)
	\le 
	\f1{\exp (\Omega(2^{k(1-2\EPSP)}k))}\,.\eeq
Combining with \eqref{e:frakH.bound}, 
and recalling the definitions of $\mathfrak{W}$ and $T$, we obtain
	\begin{align}\nonumber
	\bPa(T=1) &=
	\bPa\Bigg(\mathfrak{W}
		\le \bigg(1-\f1{k^{10}}\bigg)
		|\delta v|
		 \Bigg)\\
	&\le
	\bPa
	\Bigg(
	\max\bigg\{ \mathfrak{H} 
	,\sum_{e\in\delta v}
	\Ind{\bL_e\notin\mathbb{D}\cap\mathbb{L}}
	\bigg\} \ge
	\f{|\delta v|}{3k^{10}}\Bigg)
	\le \f1{\exp(\Omega(k^3\log k))}\,.
	\label{eq:T2Bound}
	\end{align}
Recall from \eqref{e:four.distinct.edges} the choice of four distinct edges $\bm{E}$ from $\delta v$. 
For each $e\in\bm{E}$ define
	\beq\label{e:DIVERSE.Gm}
	\Gm(e)\equiv 
	\begin{cases}
	1&\textup{if $T(\cXa)=1$, 
		$e=e^{\plus,i}$, and $x^{a,i}=\plus$,}\\
	1&\textup{if $T(\cXa)=1$, 
		$e=e^{\minus,i}$, and $x^{a,i}=\minus$,}\\
	0&\textup{otherwise.}
	\end{cases}\eeq
In the first transformation (described below), given $\cXa$ we will construct a modified configuration $\cXb$ in which $(\sigma_e)^i=\red$ and $\bL_e\in\mathbb{D}\cap\mathbb{L}$ whenever 
$e=e^{\PM,i}$ with $\Gamma(e)=1$. (This ensures that a frozen spin $x^i\in\set{\minus,\plus}$ has at least one incident $\SPIN{red}$ edge in $\delta v(x^i)$, and gives more flexibility to adjust the colors on the other incident edges.) For $e\in\bm{E}$ and $z\in\set{\minus,\plus,\free}^2$ define
	\beq\label{e:DIVERSE.p.z.e.ubd}
	p_z(e)
	\equiv\bPa(\Gamma(e)=1,x=z)
	\le \bPa(T=1)
	\stackrel{\eqref{eq:T2Bound}}{\le}
	\f{1}{\exp(\Omega(k^3\log k))}\,.
	\eeq
Note the definition of $\Gamma(e)$ implies
$p_z(e^{\plus,i})=0$ if $z^i\ne\plus$, and likewise
$p_z(e^{\minus,i})=0$ if $z^i\ne\minus$. For $e\in\bm{E}$ we also let
	\[r_z(e)
	\equiv
	\bPa\bigg(
	\bL_e\in\mathbb{D}\cap\mathbb{L},
	x=z,(\sigma_e)^i=\red,\sigma_e\ne\red\red,
	\mathfrak{R}^i\ge2, T=0
	\bigg)\,.
	\]
Note having $(\sigma_e)^i=\red$ necessitates $x^i=\plus$, so 
$r_z(e^{\plus,i})=0$ if $z^i\ne\plus$
and $r_z(e^{\minus,i})=0$ if $z^i\ne\minus$. In all other cases we claim that $r_z(e)\ge \Omega(1)/4^k$. Without loss of generality it suffices to consider the case $e=e^{\plus,1}$ and $z\in\set{\plus}\times\set{\minus,\plus,\free}$. In this case, recalling the bound
\eqref{eq:T2Bound} gives
	\[
	r_z(e)
	\ge
	\bPa\bigg(
	\bL_e\in\mathbb{D}\cap\mathbb{L},
	x^2=z^2,\sigma_e\in\set{\red}
		\times\set{\yel,\grn,\blu},
		\mathfrak{R}^i\ge2
	\bigg)
	-\f1{\exp(\Omega(k^3\log k))}\,.\]
Next recall from \eqref{eq:productForm} that $(\usi_{\delta v})^1$, $(\usi_{\delta v})^2$, and $\bL_e$ ($e\in\delta v$) are mutually independent under $\bPa$. It follows by straightforward calculations that for any $\bL$ we have
	\[
	\bPa\bigg(
	\mathfrak{R}^i\ge2
	\,\bigg|\,\bL_e=\bL,
	x^2=z^2,\sigma_e\in\set{\red}
		\times\set{\yel,\grn,\blu}\bigg)
	\ge 1-o_k(1)\,.
	\]
Combining with \eqref{eq:bEdefin} gives
	\beq\label{eq:rzBound}
	r_z(e)
	\ge
	\f{1-o_k(1)}{2}
	\bPa\bigg(
	(\sigma_e)^1=\red
	\bigg)
	\bPa\bigg(x^2=z^2,
	(\sigma_e)^2\ne\red
	\bigg)-\f1{\exp(\Omega(k^3\log k))}
	\ge \f{\Omega(1)}{4^k}\,,\eeq
as claimed. It follows by comparing \eqref{e:DIVERSE.p.z.e.ubd} with \eqref{eq:rzBound} that $p_z(e)/r_z(e)=o_k(1)$. For $e\in\bm{E}$,
we define Bernoulli random variables $\Xi(e)$ such that if $r_z(e)=0$ then $\Xi(e)\equiv0$, and otherwise if $e=e^{\PM,i}$ with $r_z(e)>0$ then
	\beq\label{e:DIVERSE.def.Xi}
	\bPa\bigg(
	\Xi(e) = 1 \,\bigg|\, x=z,
		\cXa=\cX\bigg)
	= \f{p_z(e)}{r_z(e)}
	\mathbf{1}\bigg\{
	\bL_e \in \mathbb{D}\cap \mathbb{L}, x=z,
	(\sigma_e)^i=\red,
	\sigma_e\ne\red\red, 
	\mathfrak{R}^i\geq 2,T=0
	\bigg\}\eeq
(for all possible values of $z,\cX$). It then follows from the definition of $r_z(e)$ that for all $z\in\set{\minus,\plus,\free}^2$ we have
	\beq\label{eq:GammaXiEquality2}
	\bPa\bigg(
	\Xi(e) = 1,x=z\bigg)
	=p_z(e) =\bPa\bigg(
	\Gamma(e) = 1,x=z\bigg)\,,\eeq
and consequently $\bPa(\Xi(e) = 1)=\bPa(\Gamma(e) = 1)$. Further, by the independence properties of $\bPa$ that we see from the expression \eqref{eq:productForm}, we have
	\beq\label{eq:cond.L.Xi}
	\bPa\Bigg(
		\bL_e=\bL
		\,\Bigg|\, \Xi(e) = 1, \Big(\sigma_e,
			\big(\sigma_{e'},\varsigma_{e'},\bL_{e'}\big)_{e'\in
				\delta v\setminus e}\Big)
				\Bigg)
	= \pi_{\DD}\Big(\bL\,
		\Big|\,\bt_e, \bL_e \in \mathbb{D}\cap \mathbb{L}
		\Big)\,.\eeq
This concludes our estimates on $\bPa$, and we now turn to the construction of $\bPb$.\smallskip

\noindent\bemph{\hypertarget{l:DIVERSE.Pb}{Part 3}. Construction of $\bPb$.} Recall \eqref{e:four.distinct.edges} that $\bm{E}$ is a subset of four distinct edges in $\delta v$. In the above we defined Bernoulli random variables $\Gm(e)$ and $\Xi(e)$ for $e\in\bm{E}$. Recall moreover that $\Gm(e)=1$ can only occur on the event $T=T(\cXa)=1$, while $\Xi(e)=1$ can only occur on the event $T=T(\cXa)=0$, so in particular we can never have $\Gm(e)=\Xi(e)=1$. In the construction we will essentially ``swap'' the events $\Gm(e)=1$ and $\Xi(e)=1$.
To make this precise, recall that we write $\cX_e\equiv(\sigma_e,\varsigma_e,\bL_e)$. Given $\cXa\sim\bPa$, we let $\cXb$ be defined as follows:
\begin{enumerate}[--]
\item If $e\in\delta v\setminus\bm{E}$ then set $(\cXb)_e=(\cXa)_e$.
\item If $e\in\bm{E}$ with $\Gm(e)=\Xi(e)=0$ then we also set $(\cXb)_e=(\cXa)_e$.
\item If $e\in\bm{E}$ with $\Gm(e)=1$ (hence $\Xi(e)=0$),
and $x_v=z$, then we let $(\cXb)_e$ be sampled from the law
	\beq\label{e:swap.Gamma.Xi.1}
	\bPb\bigg(\cX_e\in\cdot \,\bigg|\, x_v=z,\Gm(e)=1\bigg)
	=\bPa\bigg(\cX_e\in\cdot \,\bigg|\, x_v=z,\Xi(e)=1\bigg)\,.
	\eeq
\item If $e\in\bm{E}$ with $\Xi(e)=1$ (hence $\Gm(e)=0$),
and $x_v=z$, then we let 
$(\cXb)_e$ be sampled from the law
	\beq\label{e:swap.Gamma.Xi.2}
	\bPb\bigg(\cX_e\in\cdot \,\bigg|\, x_v=z,\Xi(e)=1\bigg)
	=\bPa\bigg(\cX_e\in\cdot \,\bigg|\, x_v=z,\Gm(e)=1\bigg)\,.
	\eeq
\end{enumerate}
This results in a valid configuration $\cXb$ which has the same frozen spin as $\cXa$, that is, $x^a=x^b\in\set{\minus,\plus,\free}^2$. For the sake of concreteness, we give two examples of the above construction:
\begin{enumerate}[(a)]
\item \textit{Example on event $T=1$.} Suppose we have $\cXa\sim\bPa$ with frozen spin $x^a=\plus\plus$, such that $T=T(\cXa)=1$. In this case it follows from the definition \eqref{e:DIVERSE.Gm} of $\Gm(e)$ that $\Gm(e^{\plus,1})=\Gm(e^{\plus,2})=1$. Since $x^a=\plus\plus$, we must have $(\sigma^a)_e\in\set{\red,\blu}^2$ for $e\in\set{e^{\plus,1},e^{\plus,2}}$. The resampling step \eqref{e:swap.Gamma.Xi.1} results in a modified configuration $\cXb$. It follows from the definition \eqref{e:DIVERSE.def.Xi} of $\Xi(e)$ that in the modified configuration,
	{\setlength{\jot}{0pt}\begin{align*}
	&\textup{$(\sigma^b)_e=\red\blu$ for $e = e^{\plus,1}$,}\\
	&\textup{$(\sigma^b)_e=\blu\red$ for $e = e^{\plus,2}$.}
	\end{align*}}%
Thus, for $e=e^{\PM,i}$ with $\Gm(e)=1$,
the procedure \eqref{e:swap.Gamma.Xi.1} makes
$(\sigma_e)^i$ marginally more likely to be $\SPIN{red}$,
but makes $\sigma_e$ less likely to be $\red\red$.

\item \textit{Example on event $T=0$.} 
Suppose we have $\cXa\sim\bPa$ with frozen spin $x^a=\plus\plus$, such that $T=T(\cXa)=0$. Recalling \eqref{e:DIVERSE.def.Xi}, suppose that we have $\Xi(e^{\plus,1})=1$
while $\Xi(e^{\plus,2})=0$. In particular, this means for $e=e^{\plus,1}$ we must have $(\sigma^a)_e=\red\blu$. The resampling step \eqref{e:swap.Gamma.Xi.2} results in a modified configuration $\cXb$. It follows from the definition \eqref{e:DIVERSE.Gm} of $\Gm(e)$ that in the modified configuration we must have $(\sigma^b)_e \in \set{\red,\blu}^2$
for $e=e^{\plus,1}$. Thus, for $e=e^{\PM,i}$ with $\Xi(e)=1$,
the procedure \eqref{e:swap.Gamma.Xi.2}
makes $(\sigma_e)^i$ marginally less likely to be $\SPIN{red}$,
but makes $\sigma_e$ more likely to be $\red\red$.
\end{enumerate}
We let $\P$ denote the joint law of $(\cXa,\cXb)$, and let $\bPb$ denote the marginal law of $\cXb$. Note that $\bPa$ and $\bPb$ have the same edge marginals: for $e\in\delta v\setminus \bm{E}$ it is clear that $(\cXa)_e$ and $(\cXb)_e$ have the same marginal law, since in the coupling we set $(\cXb)_e=(\cXa)_e$.
For $e\in\bm{E}$, we have
	{\setlength{\jot}{0pt}\begin{align}\nonumber
	&\bPb(\cX_e\in\cdot\,|\, x_v=z,
		\Gm(e)+\Xi(e)=1) \\ \nonumber
	&=\bPa(\Gm(e)=1\,|\,x_v=z)
	\bPb(\cX_e\in\cdot\,|\, x_v=z,\Gm(e)=1)
	+\bPa(\Xi(e)=1\,|\,x_v=z)
	\bPb(\cX_e\in\cdot\,|\, x_v=z,\Xi(e)=1)\\ \nonumber
	&=
	\bPa(\Xi(e)=1\,|\,x_v=z)
	\bPa(\cX_e\in\cdot\,|\, x_v=z,\Xi(e)=1)
	+\bPa(\Gm(e)=1\,|\,x_v=z)
	\bPa(\cX_e\in\cdot\,|\, x_v=z,\Gm(e)=1)\\
	&=\bPa(\cX_e\in\cdot\,|\, x_v=z,
		\Gm(e)+\Xi(e)=1)\,,
	\label{e:DIVERSE.ab.same.mgls}
	\end{align}}%
where the transition to the third line uses \eqref{eq:GammaXiEquality2}, \eqref{e:swap.Gamma.Xi.1}, and \eqref{e:swap.Gamma.Xi.2}. This verifies that $\bPb(\cX_e\in\cdot)=\bPa(\cX_e\in\cdot)$ for all $e\in\delta v$, so the measure $\bPb$ again satisfies the constraints \eqref{e:D.L.diverse.clause}, \eqref{e:D.L.diverse.judicious}, and \eqref{eq:relaxedZetaConstraint}.
We now compare the entropy of $\bPa$ and $\bPb$. As above, denote $T=T(\cXa)$. Write $\bm{E}_a\equiv\set{e\in\bm{E}:\Gm(e)+\Xi(e)=1}$. Denote $\cX\equiv(\cY,\cZ)$ where
	{\setlength{\jot}{0pt}\begin{align*}
	\cY&\equiv (\cX_e : e\in\bm{E}_a)\,,\\
	\cZ&\equiv (\cX_e : e\in\delta v\setminus\bm{E}_a)\,.
	\end{align*}}%
If $\bm{E}_a=\emptyset$ then $\cY$ is the null vector. Likewise denote $\uX_{\delta v}\equiv (Y,Z)$
where $Y\equiv (X_e:e\in\bm{E}_a)$. We then have
	{\setlength{\jot}{0pt}\begin{align}\nonumber
	\Ent(\cXb) 
	&\ge\Ent(T,\bm{E}_a,\cZ^b)
		+\Ent( \cY^b\,|\, T,\bm{E}_a,\cZ^b)
		-\Ent(T)-\Ent(\bm{E}_a)\\
	&=\Ent(T,\bm{E}_a,\cZ^b)
		+\Ent( (\uL^b)_{\bm{E}_a}\,|\, T,\bm{E}_a,\cZ^b)
		+\Ent( Y^b\,|\, T,\bm{E}_a,\cZ^b,(\uL^b)_{\bm{E}_a})
		-\Ent(T)-\Ent(\bm{E}_a)
	\label{eq:EntDecompB}\,,\\ \nonumber
	\Ent(\cXa) 
	&\le \Ent(T,\bm{E}_a,\cZ^a) +
		\Ent(\cY^a\,|\,T,\bm{E}_a,\cZ^a)\\
	&= \Ent(T,\bm{E}_a,\cZ^a)
		+\Ent((\uL^a)_{\bm{E}_a}\,|\,T,\bm{E}_a,\cZ^a)
		+\Ent(Y^a\,|\,T,\bm{E}_a,\cZ^a,(\uL^a)_{\bm{E}_a})
		\,.
	\label{eq:EntDecompA}
	\end{align}}%
Recall that the edges outside $\bm{E}_a$ are left unchanged by the transformation from $\cXa$ to $\cXb$, so
	\beq\label{e:DIVERSE.non.swapped.part}
	\Ent\bigg(T,\bm{E}_a,\cZ^b\bigg)
	=\Ent\bigg(T,\bm{E}_a,\cZ^a\bigg)\,.\eeq
Next, \eqref{eq:T2Bound} and \eqref{eq:GammaXiEquality2} together imply that
	\[
	\bPa(\bm{E}_a\ne\emptyset) 
	\le O\Big( \bPa(T=1)\Big)
	\le O\bigg( \f{1}{\exp(\Omega(k^3\log k))}
		\bigg)\,,\]
from which it follows that
	\beq\label{eq:sigma.T.Ent}
	\max\bigg\{\Ent(T),\Ent(\bm{E}_a),
		\Ent(Y^a\,|\,\bm{E}_a),
		\Ent(Y^b\,|\,\bm{E}_a)
		\bigg\} = o_k(1)\,.\eeq
We emphasize that in \eqref{eq:sigma.T.Ent} we must put $Y^a$ and $Y^b$ rather than $\cY^a$ and $\cY^b$:
this is because $Y$ consists of spins
$X_e$ which take only $O(1)$ possibilities, while $\cY$ consists of spins $\cX_e=(X_e,\bL_e)$ which take a large number of possibilities (indeed, growing with the neighborhood radius $R$ of \eqref{e:radii}) and therefore can have large conditional entropy given $\bm{E}_a$. To deal with the clause types $\bL_e$ we note that
	\begin{align}\nonumber
	&\Ent\bigg( 
	(\uL^a)_{\bm{E}_a}
	\,\bigg|\,
	\bm{E}_a,T, \cZ^a\bigg) \\ \nonumber
 	&\stackrel{\eqref{eq:cond.L.Xi}}{=}
	\sum_{e\in \bm{E}}
	\Bigg\{
	\bPa(\Gamma(e)=1)
		\Ent\Big(\pi_{\DD}(\cdot\,|\,\bt_e)\Big)
	+\bPa(\Xi(e)=1)
	\Ent\Big(\pi_{\DD}(\cdot
	\,|\,\bt_e, \bL_e \in \mathbb{D}\cap\mathbb{L})\Big)\\ \nonumber
	&\stackrel{\eqref{eq:GammaXiEquality2}}{=} \sum_{e\in \bm{E}}
	\Bigg\{
	\bPb(\Gamma(e)=1)
		\Ent\Big(\pi_{\DD}(\cdot\,|\,\bt_e)\Big)
	+\bPb(\Xi(e)=1)
	\Ent\Big(\pi_{\DD}(\cdot
	\,|\,\bt_e, \bL_e \in \mathbb{D}\cap \mathbb{L})\Big)
	\Bigg\}\\
	&\stackrel{\eqref{eq:cond.L.Xi}}{=}
	\Ent\bigg( 
	(\uL^b)_{\bm{E}_a}
	\,\bigg|\,
	\bm{E}_a,T,\cZ^b
	\bigg)\,. \label{eq:L.EntopyEquality}
	\end{align}
Combining \eqref{eq:sigma.T.Ent} and \eqref{eq:L.EntopyEquality} gives 
	\begin{align}\nonumber
	&\bigg| \Ent\Big( \cY^a 
		\,\Big|\, \bm{E}_a,T,\cZ^a\Big)- 
	\Ent\Big( \cY^b
		\,\Big|\, \bm{E}_a,T,\cZ^b\Big)\bigg| \\ \nonumber
	&=
	\bigg|\Ent\Big( (\uL^a)_{\bm{E}_a}\,\Big|\, \bm{E}_a,T,\cZ^a\Big)
	-\Ent\Big( (\uL^b)_{\bm{E}_a}\,\Big|\, \bm{E}_a,T,\cZ^b\Big)\\
	\nonumber
	&\qquad+\Ent\Big( Y^a \,\Big|\, \bm{E}_a,T,\cZ^a,(\uL^a)_{\bm{E}_a}\Big)
	-\Ent\Big( Y^b \,\Big|\, \bm{E}_a,T,\cZ^b,(\uL^b)_{\bm{E}_a}\Big)
	\bigg|\\ 
	&\stackrel{\eqref{eq:L.EntopyEquality}}{=}
	\bigg|\Ent\Big( Y^a \,\Big|\, \bm{E}_a,T,\cZ^a,(\uL^a)_{\bm{E}_a}\Big)
	-\Ent\Big( Y^b \,\Big|\, \bm{E}_a,T,\cZ^b,(\uL^b)_{\bm{E}_a}\Big)\bigg|
	\stackrel{\eqref{eq:sigma.T.Ent}}{=} o_k(1)\,.
	\label{eq:EprimeEntopy}
	\end{align}
Finally, combining \eqref{eq:EntDecompB}, \eqref{eq:EntDecompA}, \eqref{e:DIVERSE.non.swapped.part},
\eqref{eq:sigma.T.Ent}, and \eqref{eq:EprimeEntopy} gives
	\beq\label{eq:A.to.B.Ent.Comp}
	\Ent(\bPb)-\Ent(\bPa)
	=\Ent(\cXb) - \Ent(\cXa)
	\ge o_k(1)\,.
	\eeq
This concludes our analysis of the measure $\bPb$.\smallskip

\noindent\bemph{\hypertarget{l:DIVERSE.Pc}{Part 4}. Construction of $\bPc$.}
Given a configuration $\cX$, for 
$z\in\set{\minus,\plus}$ and
$i=1,2$ let
	\beq\label{e:DIVERSE.hat.E}
	\hat{\bm{E}}^{z,i}
	\equiv \hat{\bm{E}}^{z,i}(\cX)
	\equiv \bigg\{ e\in\delta v(z):
	(\sigma_e)^i=\red,\sigma_e\ne\red\red,
	\bL_e\in\mathbb{D}\cap\mathbb{L}\bigg\}\,.
	\eeq
Write $\hat{\bm{E}}^i\equiv
\hat{\bm{E}}^{\minus,i}\cup\hat{\bm{E}}^{\plus,i}$.
Note that $\hat{\bm{E}}^1\cap\hat{\bm{E}}^2=\emptyset$. Recall from \eqref{e:four.distinct.edges} the definition of $\bm{E}$, and for $z\in\set{\minus,\plus}$ let
	\[
	\acute{\bm{E}}^z
	\equiv \acute{\bm{E}}^z(\cX)
	\equiv\bigg\{e\in\delta v(z)\setminus\bm{E}
	: \vsi_e=\whi\whi,\bL_e\in\mathbb{D}\cap\mathbb{L}\bigg\}\,.
	\]
(The sets $\acute{\bm{E}}^\PM(\cXa)$
are disjoint from $\bm{E}$, but
the sets $\hat{\bm{E}}^i$ can intersect $\bm{E}$.) We construct $\cXc$ as follows:
\begin{enumerate}[(II-2)]
\item[\hypertarget{i:DIVERSE.I1}{(I-1)}] On the event $\set{\cXa=\cXb,T(\cXa)=0,|\hat{\bm{E}}^1(\cXa)|\ge2}$,
choose $\hat{e}^1$ uniformly at random from $\hat{\bm{E}}^1(\cXa)$. For $e=\hat{e}^1$, set 
$(\cXc)_e=(\cXb)_e$ with probability $1-2^{-k\EPSP/2}$. With the remaining probability $2^{-k\EPSP/2}$ let $(\cXc)_e$ 
for $e=\hat{e}^1$
be defined by
	{\setlength{\jot}{0pt}\begin{align*}
	(\sigma^{c,1},\varsigma^{c,1})_e
	&=(\blu,\whi)\,,\\
	(\sigma^{c,2},\varsigma^{c,2},\bL^c)_e
	&=(\sigma^{b,2},\varsigma^{b,2},\bL^b)_e
	\end{align*}}%
Let $p^{\textup{I},1}(e,\bL)$ denote the probability that edge $e$ has clause type $\bL$ and is changed by the above:
	\beq\label{e:DIVERSE.def.p.I1}
	p^{\textup{I},1}(e,\bL)
	= \f1{2^{k\EPSP/2}}
	\E
	\Bigg(
	\f{\Ind{e\in\hat{\bm{E}}^1(\cXa)}}
		{|\hat{\bm{E}}^1(\cXa)|};
	(\bL^a)_e=\bL,
	\cXa=\cXb,T(\cXa)=0,
	|\hat{\bm{E}}^1(\cXa)|\ge2
	\Bigg)\,,\eeq
where $\E$ denotes expectation with respect to the joint law $\P$ of $(\cXa,\cXb)$. 

\item[\hypertarget{i:DIVERSE.I2}{(I-2)}]
On the event $\set{\cXa=\cXb,T(\cXa)=0,|\hat{\bm{E}}^2(\cXa)|\ge2}$, perform the analogue of 
\hyperlink{i:DIVERSE.I1}{(I-1)} switching the roles of the two copies, and define the corresponding probability $p^{\textup{I},2}(e,\bL)$.

\item[\hypertarget{i:DIVERSE.II1}{(II-1)}] 
On the event $\set{x^{a,1}\in\set{\minus,\plus},
T(\cXa)=0,|\hat{\bm{E}}^1(\cXa)|\le1}$, choose $\acute{e}^1$ uniformly at random from $\acute{\bm{E}}^{x^1}(\cXa)$. (Note the condition $T(\cXa)=0$ guarantees that the sets $\acute{\bm{E}}^{\PM}(\cXa)$ are both large, so it is possible to choose $\acute{e}^1$.) Let $(\cXc)_e$ for $e=\acute{e}^1$ be defined by
	{\setlength{\jot}{0pt}\begin{align*}
	(\sigma^{c,1},\varsigma^{c,1})_e
	&= (\red,\red)\,,\\
	(\sigma^{c,2},\varsigma^{c,2},\bL^c)_e
		&=(\sigma^{b,2},\varsigma^{b,2},\bL^b)_e
		\,.
	\end{align*}}%
Let $p^{\textup{II},1}(e,\bL)$ denote the probability that edge $e$ has clause type $\bL$ and is changed by the above:
	\beq\label{e:DIVERSE.def.p.II1}
	p^{\textup{II},1}(e,\bL)
	= \sum_{z\in\set{\minus,\plus}}
	\E\Bigg(
	\f{\Ind{e\in\acute{E}^z(\cXa)}}
		{|\acute{E}^z(\cXa)|}
	; (\bL^a)_e=\bL, x^{a,1}=z,
	T(\cXa)=0,
	|\hat{\bm{E}}^1(\cXa)|\le1
	\Bigg)\,,
	\eeq
where again $\E$ denotes expectation with respect to the joint law $\P$ of $(\cXa,\cXb)$. 
\item[\hypertarget{i:DIVERSE.II2}{(II-2)}] 
Lastly, on the event
$\set{x^{a,2}\in\set{\minus,\plus},T(\cXa)=0,|\hat{\bm{E}}^2(\cXa)|\le1}$,
perform the analogue of \hyperlink{i:DIVERSE.II1}{(II-1)} switching the roles of the two copies, with the modification that if \hyperlink{i:DIVERSE.II1}{(II-1)} already occurred, then $\acute{e}^2$ is chosen uniformly at random from $\acute{\bm{E}}(\cXa)^2\setminus\set{\acute{e}^2}$. Define the corresponding probability $p^{\textup{II},2}(e,\bL)$.
\end{enumerate}
We point out that
steps \hyperlink{i:DIVERSE.I1}{(I-1)}
and \hyperlink{i:DIVERSE.I2}{(I-2)}
can only occur if $\cXa=\cXb$,
but steps 
\hyperlink{i:DIVERSE.II1}{(II-1)}
and \hyperlink{i:DIVERSE.II2}{(II-2)}
can occur even in the case
$\cXa\ne\cXb$. Note also that \hyperlink{i:DIVERSE.I1}{(I-1)} and \hyperlink{i:DIVERSE.II1}{(II-1)}
 cannot both occur, and likewise \hyperlink{i:DIVERSE.I2}{(I-2)} and 
 \hyperlink{i:DIVERSE.II2}{(II-2)}
 cannot both occur. To finish the construction, for all edges $e\in\delta v$ which were not chosen as $\hat{e}^i$ or $\acute{e}^i$ in any of the above steps, set $(\cXc)_e=(\cXb)_e$. Note that the frozen spin has not changed, $x=x^a=x^b=x^c\in\set{\minus,\plus,\free}^2$. Moreover we have $(\bL^b)_e=(\bL^c)_e$ for all $e\in\delta v$. It is straightforward to verify that the resulting $\cXc$ is a valid configuration with the same frozen spin, $x^b=x^c\in\set{\minus,\plus,\free}^2$.
Moreover, we observe that the construction so far guarantees 
	\beq\label{e:DIVERSE.hat.E.nonempty} 
	\hat{\bm{E}}^{z,i}(\cXc)\ne\emptyset
	\quad\textup{whenever}\quad
	x^{c,i}=z\in\set{\minus,\plus}\,,
	\eeq
for $\hat{\bm{E}}^{z,i}$ as defined by \eqref{e:DIVERSE.hat.E}. Indeed, if $x=x^c$
has in the first copy $x^1=z\in\set{\minus,\plus}$, we have these cases:
\begin{enumerate}[--]
\item Suppose $T(\cXa)=1$. It follows from \eqref{e:DIVERSE.Gm} that we will have $\Gm(e)=1$ for $e=e^{z,1}\in\bm{E}$. In this case the resampling step 
\eqref{e:swap.Gamma.Xi.1} and the definition 
\eqref{e:DIVERSE.def.Xi}
of $\Xi(e)$ results in
a modified configuration $\cXb$ such that
 $e$ belongs to $\hat{\bm{E}}^{z,1}(\cXb)$.
(If $x^2=y\in\set{\minus,\plus}$ then we may also have $\Gm(e^{y,2})=1$, in which case the set $\hat{\bm{E}}^{y,2}$ also increases going from $\cXa$ to $\cXb$, but this is not relevant to the current considerations.)
Then steps \hyperlink{i:DIVERSE.I1}{(I-1)}
and \hyperlink{i:DIVERSE.I2}{(I-2)} do not occur
because $\cXa\ne\cXb$,
and steps 
\hyperlink{i:DIVERSE.II1}{(II-1)}
and \hyperlink{i:DIVERSE.II2}{(II-2)}
do not occur because $T(\cXa)=1$.
It follows that 
	\[|\hat{\bm{E}}^{z,1}(\cXc)|
	=|\hat{\bm{E}}^{z,1}(\cXb)|
	\ge|\hat{\bm{E}}^{z,1}(\cXa)|+1
	\ge1\,,\]
since $e=e^{z,1}\in\hat{\bm{E}}^{z,1}(\cXb)\setminus 
\hat{\bm{E}}^{z,1}(\cXa)$.

\item Suppose $T(\cXa)=0$, $\cXa=\cXb$, and $|\hat{\bm{E}}^1(\cXa)|\ge2$. In this case, the only way for an edge $e$ to belong to $\hat{\bm{E}}^1(\cXa)=\hat{\bm{E}}^1(\cXb)$ and not to $\hat{\bm{E}}^1(\cXc)$ is if $e$ is chosen as $\hat{e}^1$ in step \hyperlink{i:DIVERSE.I1}{(I-1)}. It follows that
	\[
	|\hat{\bm{E}}^1(\cXc)|
	\ge |\hat{\bm{E}}^1(\cXb)|-1
	= |\hat{\bm{E}}^1(\cXa)|-1
	\ge1\,.
	\]

\item Suppose $T(\cXa)=0$, $\cXa\ne\cXb$, and $|\hat{\bm{E}}^1(\cXa)|\ge2$. 
In this case, we see from \eqref{e:DIVERSE.def.Xi}
that we may have $\Xi(e)=1$
for $e=e^{z,1}$. This means $e\in\hat{\bm{E}}^1(\cXa)$, and the resampling step \eqref{e:swap.Gamma.Xi.2} results in $\cXb$ such that $e\notin\hat{\bm{E}}^1(\cXb)$. (If $x^2=y\in\set{\minus,\plus}$ then we may also have $\Xi(e^{y,2})=1$, in which case the set 
$\hat{\bm{E}}^{y,2}$
also decreases going from $\cXa$ to $\cXb$, but this is not relevant to the current considerations since the $\hat{\bm{E}}^i$ are disjoint.) Then steps \hyperlink{i:DIVERSE.I1}{(I-1)}
and \hyperlink{i:DIVERSE.I2}{(I-2)} do not occur
because $\cXa\ne\cXb$, and steps
\hyperlink{i:DIVERSE.II1}{(II-1)}
and \hyperlink{i:DIVERSE.II2}{(II-2)}
do not occur because $|\hat{\bm{E}}^1(\cXa)|\ge2$.
It follows that
	\[
	|\hat{\bm{E}}^1(\cXc)|
	=|\hat{\bm{E}}^1(\cXb)|
	\ge|\hat{\bm{E}}^1(\cXa)|-1\ge1\,.
	\]

\item Suppose $T(\cXa)=0$ and $|\hat{\bm{E}}^1(\cXa)|\le1$. In this case, step 
\hyperlink{i:DIVERSE.II1}{(II-1)}
occurs and results in 
a configuration $\cXc$ such that
$\hat{\bm{E}}^1(\cXc)$ contains $\hat{e}^1$, and so is nonempty.
\end{enumerate}
This verifies the property \eqref{e:DIVERSE.hat.E.nonempty}, which will be used
in later steps of the proof. We now let $\P$ denote the joint law of
 $(\cXa,\cXb,\cXc)$, and let $\bPc$ denote the marginal law of $\cXc$. 

We first estimate the probability $p^{\textup{I},1}(e,\bL)$ from \eqref{e:DIVERSE.def.p.I1}. For an upper bound, note in order for $e\in\hat{\bm{E}}^1(\cX)$ we must have $(\sigma_e)^1=\red$, so for all clause types $\bL$ we have
	\beq\label{e:DIVERSE.p.I1.ubd.indiv}
	p^{\textup{I},1}(e,\bL)
	\le
	\f{\bPa( \bL_e=\bL, (\sigma_e)^1=\red )}
		{2^{k\EPSP/2}}
	= \f{\pi_{\DD}(\bL\,|\,\bt_e)
		\cdot \starpi_e(\red)}
		{2^{k\EPSP/2}}\,.\eeq
Next, since the total probability that \hyperlink{i:DIVERSE.I1}{(I-1)} occurs is at most $2^{-k\EPSP/2}$, we can also bound
	\beq\label{e:DIVERSE.p.I1.ubd}
	\sum_{e\in\delta v}\sum_{\bL} 
	p^{\textup{I},1}(e,\bL)
	\le \f1{2^{k\EPSP/2}}\,.\eeq
Next we note that $p^{\textup{I},1}(e,\bL)=0$ if $\bL\notin\mathbb{D}\cap\mathbb{L}$. In the case $\bL\in\mathbb{D}\cap\mathbb{L}$ we will derive a lower bound on $p^{\textup{I},1}(e,\bL)$. Recall $R_e$ from \eqref{e:DIVERSE.Re}, and let $S_e$ be an independent copy of $R_e$. By similar considerations as for \eqref{eq:frakR.bound}, we have
	\beq\label{e:DIVERSE.tildefrakR.intermediate.rr}
	\bPa\Bigg(
	\ddot{\mathfrak{r}}(\cX)
	\equiv
	\sum_{e\in\delta v} \Ind{\sigma_e=\red\red}
	\ge k^{1/2}\Bigg)
	\le O(1)
	\max_{z\in\set{\minus,\plus}}\Bigg\{
	\f{\P(\sum_{e\in\delta v(z)}
		R_e S_e \ge k^{1/2})}
	{\P(\sum_{e\in\delta v(z)} R_e\ge1)^2}\Bigg\}
	\le \f1{\exp(\Omega(k^{3/2}))}\,.\eeq
Next, for $e\in\delta v$, let $G_e$ denote independent Bernoulli random variables with
	\[\P(G_e=1)
	=\bPa(\bL_e\in\mathbb{D}\cap\mathbb{L})
	=\sum_{\bL}
		\pi_{\DD}(\bL\,|\,\bt_e)
		\Ind{\bL\in\mathbb{D}\cap\mathbb{L}}\,.\]
For $z\in\set{\minus,\plus}$ and any fixed $e\in\delta v(z)$, on the event $\set{(\sigma_e)^1=\red}$ we have
	\beq\label{e:DIVERSE.tilde.r.i.lbd}
	(\mathfrak{r}_e)^1(\cX)
	\equiv
	\sum_{e'\in\delta v(z)\setminus e}
	\mathbf{1}\bigg\{(\sigma_{e'})^1=\red,
		\bL_{e'} \in \mathbb{D}\cap\mathbb{L}
	\bigg\}
	\ge \Big|\hat{\bm{E}}^1(\cX)\setminus
		\set{e}\Big|
	\ge (\mathfrak{r}_e)^1(\cX) 
	- \ddot{\mathfrak{r}}(\cX)\,,\eeq
where $\ddot{\mathfrak{r}}(\cX)$ was defined on the left-hand side of \eqref{e:DIVERSE.tildefrakR.intermediate.rr}. By similar considerations as for \eqref{eq:frakR.bound} and \eqref{e:DIVERSE.tildefrakR.intermediate.rr}, we can express
	\[
	\bPa\Bigg(
	(\mathfrak{r}_e)^1 \le \f{k}{3}
	\,\Bigg|\, (\sigma_e)^1=\red\Bigg)
	=\P\Bigg(
	 \sum_{e'\in\delta v(z)\setminus e}
	 R_{e'} G_{e'} \le \f{k}{3}\Bigg)\,.
	\]
To bound this last probability, note that
for both $z\in\set{\minus,\plus}$ we have
	\begin{align*}
	&\Var\Bigg( \sum_{e'\in\delta v(z)\setminus e}
	R_{e'} G_{e'}\Bigg)
	\le \E\Bigg( \sum_{e'\in\delta v(z)\setminus e}
	R_{e'} G_{e'}\Bigg)
	=\sum_{e'\in\delta v(z)\setminus e}
	\E R_{e'}
	-\sum_{e'\in\delta v(z)\setminus e}
	(\E R_{e'}) \E(1-G_{e'})\\
	&\qquad=
	\bigg(1-o_k(1)\bigg) \f{k\log 2}{2}
	-\sum_{e'\in\delta v(z)\setminus e}
	(\E R_{e'}) \E(1-G_{e'}) 
	\stackrel{\eqref{e:assumption.notD.notL}}{=}
	\bigg(1-o_k(1)\bigg) \f{k\log 2}{2}\,.
	\end{align*}
We can use this last estimate together with Bernstein's inequality to bound
	\beq\label{e:DIVERSE.bernstein}
	\bPa\Bigg(
	(\mathfrak{r}_e)^1 \le \f{k}{3}
	\,\Bigg|\, (\sigma_e)^1=\red\Bigg)
	=\P\Bigg(
	 \sum_{e'\in\delta v(z)\setminus e}
	 R_{e'} G_{e'} \le \f{k}{3}\Bigg)
	\le \f1{\exp(\Omega(k))}\,.\eeq
On the other hand, by similar arguments as for \eqref{eq:frakR.bound}, we can bound
	\[
	\bPa\Bigg(
	(\mathfrak{r}_e)^1 \ge k
	\,\Bigg|\, (\sigma_e)^1=\red\Bigg)
	\le \P\Bigg( \sum_{e'\in\delta v(z)\setminus e}
		R_{e'} \ge k\Bigg)
	\le \f1{\exp(\Omega(k\log k))}\,.
	\]
Combining the last two bounds gives
	\beq\label{e:DIVERSE.tildefrakR.intermediate.red}
	\bPa\Bigg(
	(\mathfrak{r}_e)^1 \in
		\bigg[\f{k}{3},k\bigg]
	\,\Bigg|\, (\sigma_e)^1=\red\Bigg)
	\ge1- \f1{\exp(\Omega(k))}\,.
	\eeq
Returning to \eqref{e:DIVERSE.def.p.I1}, 
for $\bL\in\mathbb{D}\cap\mathbb{L}$ we can lower bound
	\begin{align*}
	p^{\textup{I},1}(e,\bL) 
	&\ge \f1{2^{k\EPSP/2} \cdot 2k}
	\bPa\Bigg(
	\bL_e=\bL,
	(\sigma_e)^1=\red,
	(\sigma_e)^2\ne\red,
	(\mathfrak{r}_e)^1 \in \bigg[\f{k}{3},k\bigg]
	\Bigg)\\
	&\qquad-\Bigg\{\bPa(T=1)
	+\P(\cXa\ne \cXb)
	+\bPa(\ddot{\mathfrak{r}}\ge k^{1/2})
		\Bigg\}\,.\end{align*}
In the above we used the assumption that $\bL\in\mathbb{D}\cap\mathbb{L}$, so if
$(\sigma_e)^1=\red$ and $(\sigma_e)^2\ne\red$
then $e\in\hat{\bm{E}}^1$.
We also used that if $(\mathfrak{r}_e)^1 \ge k/3$
and $\ddot{\mathfrak{r}}\le k^{1/2}$, then
certainly \eqref{e:DIVERSE.tilde.r.i.lbd} implies $|\hat{\bm{E}}^1|\ge2$, as is required for step (I-1) to occur. Lastly we used that if
$(\mathfrak{r}_e)^1 \le k$, then 
it follows from \eqref{e:DIVERSE.tilde.r.i.lbd} that $|\hat{\bm{E}}^1|\le k+1$, so edge $e$ has probability at least $1/(k+1) \ge 1/(2k)$ to be chosen as $\hat{e}^1$. Now recall from the previous construction that
$\P(\cXa\ne \cXb)
\le O(\bPa(T=1))$, so applying \eqref{eq:T2Bound} and \eqref{e:DIVERSE.tildefrakR.intermediate.rr} gives
	\begin{align*}
	p^{\textup{I},1}(e,\bL) 
	\ge \f1{2^{k\EPSP/2} \cdot 2k}
	\bPa\Bigg(
	\bL_e=\bL,
	(\sigma_e)^1=\red,
	(\sigma_e)^2\ne\red,
	(\mathfrak{r}_e)^1 \in \bigg[\f{k}{3},k\bigg]
	\Bigg)
	-\f1{\exp(\Omega(k^{3/2}))}\,.
	\end{align*}
It then follows (again using the independence properties of $\bPa$) that
	\begin{align}\nonumber
	p^{\textup{I},1}(e,\bL) 
	&\stackrel{\eqref{eq:productForm}}{\ge}
	\f{\pi_{\DD}(\bL\,|\,\bt_e)}
		{2^{k\EPSP/2} \cdot 2k}
	\bPa\Big((\sigma_e)^1=\red\Big)
	\bPa\Big((\sigma_e)^2\ne\red\Big)
	\bPa\Bigg(
	(\mathfrak{r}_e)^1 \in \bigg[\f{k}{3},k\bigg]
	\,\Bigg|\,
	(\sigma_e)^1=\red
	\Bigg) \\
	&\stackrel{\eqref{e:DIVERSE.tildefrakR.intermediate.red}}{\ge}
	\f{\pi_{\DD}(\bL\,|\,\bt_e)}
		{2^{k\EPSP/2} \cdot 2k}
	\cdot \f{1-o_k(1)}{2^k}\,,
	\label{e:DIVERSE.p.I.final.lbd}\end{align}
as long as $\bL\in\mathbb{D}\cap\mathbb{L}$. This concludes our analysis of $p^{\textup{I},1}(e,\bL)$. Clearly, 
\eqref{e:DIVERSE.p.I1.ubd} and \eqref{e:DIVERSE.p.I.final.lbd} also hold for 
$p^{\textup{I},2}(e,\bL)$.

We next prove an upper bound on the probability 
$p^{\textup{II},1}(e,\bL)$ from
\eqref{e:DIVERSE.def.p.II1}. Recall that the condition $T(\cXa)=0$ implies 
$|\acute{\bm{E}}^z(\cXa)| = \Theta(k2^k)$
for both $z\in\set{\minus,\plus}$.
It follows that we can upper bound \eqref{e:DIVERSE.def.p.II1} as
	\[
	p^{\textup{II},1}(e,\bL)
	\le
	O\bigg(\f1{k2^k}\bigg)
	\cdot
	\bPa\bigg(\bL_e=\bL,
	|\hat{\bm{E}}^1(\cXa)|\le1
	\bigg)
	\le
	O\bigg(\f{\pi_{\DD}(\bL\,|\,\bt_e)}{k2^k}\bigg)
	\cdot\bPa\bigg(
		|\hat{\bm{E}}^1(\cXa)
		\setminus\set{e}|\le1
		\bigg)\,.\]
Recall from
\eqref{e:DIVERSE.tilde.r.i.lbd} 
that $|\hat{\bm{E}}^1(\cXa)\setminus\set{e}| \le (\mathfrak{r}_e)^1(\cXa)$. Then, arguing similarly as for 
\eqref{e:DIVERSE.bernstein}, 
for each $z\in\set{\minus,\plus}$
we can apply Bernstein's inequality to upper bound
	\begin{align*}
	&\bPa\Bigg(
	(\mathfrak{r}_e)^1
	=
	\sum_{e'\in\delta v(z)\setminus e}
	\mathbf{1}\Big\{(\sigma_{e'})^1=\red,
	\bL_{e'}\in\mathbb{D}\cap\mathbb{L}
	\Big\} \le1
	\,\Bigg|\, x^1=z\Bigg) \\
	&\qquad= \f{\P(\sum_{e'
		\in\delta v(z)\setminus e} R_{e'} G_{e'}
			\le1)}
		{\P(\sum_{e'\in\delta v(z)} R_{e'}\ge1)}
	\le \f1{\exp(k c_0)}\,,
	\end{align*}
for a positive absolute constant $c_0$.
On the other hand, if $(\mathfrak{r}_e)^1\ge2$, in order for $|\hat{\bm{E}}^1(\cXa)\setminus\set{e}|\le1$ it must be that
$(\sigma^{a,2})_{e'}=\red$ for all but at most one of the edges $e'$ contributing to $(\mathfrak{r}_e)^1$, which has chance $O(2^{-k})$. Then
	\[
	\bPa\bigg(
		\hat{\bm{E}}^1(\cXa)
		\setminus\set{e}=\emptyset\bigg)
	\le
	\bPa \bigg((\mathfrak{r}_e)^1\le1\bigg)
	+\underbrace{\bPa\bigg(
	(\mathfrak{r}_e)^1\ge2,
	|\hat{\bm{E}}^1(\cXa)
		\setminus\set{e}| \le1
		\bigg)}_{O(1/2^k)}
	\le \f1{\exp(k c_0)}\,.\]
Substituting into the above expression for
$p^{\textup{II},1}(e,\bL)$ gives
	\beq\label{e:DIVERSE.p.II.final.ubd}
	p^{\textup{II},1}(e,\bL)
	\le \f{\pi_{\DD}(\bL\,|\,\bt_e)}
		{2^{k(1+c_0)}}\,.
	\eeq
A similar bound holds for $p^{\textup{II},2}(e,\bL)$, and this concludes our analysis of $p^{\textup{II},1}(e,\bL)$.

We now turn to estimating the entropy of $\bPc$. Let $\bm{E}_b$ be the (random) subset of edges $e\in\delta v$ for which $(\cXb)_e\ne(\cXc)_e$, or equivalently $(X^b)_e\ne(X^c)_e$. It follows by combining \eqref{e:DIVERSE.p.I1.ubd} and \eqref{e:DIVERSE.p.II.final.ubd} that
	\beq\label{e:DIVERSE.expected.bc.change}
	\E(|\bm{E}_b|)
	\le \sum_{e\in\delta v}\sum_{\bL}\sum_{i=1,2}
		\bigg(
		p^{\textup{I},i}(e,\bL)
		+p^{\textup{II},i}(e,\bL)\bigg)
	\le O\bigg( \f1{2^{k\EPSP/2}}\bigg)\,.\eeq
We write $\Ent(\bm{E}_b)$ to denote the entropy of the random variable $\bm{E}_b$, and we write $H(p)\equiv -p\log p-(1-p)\log (1-p)$ to denote the binary entropy function evaluated at $p\in[0,1]$. Then
	\beq\label{e:DIVERSE.bc.ent.Estar}
	\Ent(\bm{E}_b)
	\le \sum_{e\in\delta v} \Ent( \Ind{e\in\bm{E}_b} )
	\le |\delta v|
	H\bigg( \f{ \E(|\bm{E}_b|) }{|\delta v|}\bigg)
	\stackrel{\eqref{e:DIVERSE.expected.bc.change}}{=} 
	o_k(1)\,,
	\eeq
where the second step used Jensen's inequality and the concavity of the binary entropy function. Now decompose $\uX_{\delta v}\equiv (U,W)$ where $U\equiv(X_e:e\in\bm{E}_b)$. 
Note if $\bm{E}_b$ is given, then $U^b$ takes at most $O(1)$ distinct values, so it follows from \eqref{e:DIVERSE.bc.ent.Estar} that the conditional entropy of $U^b$ given $\bm{E}_b$ is small, that is,
$\Ent(U^b \,|\, \bm{E}_b)=o_k(1)$. Then, since $\cX^b$ can be determined as a function of $(\cX^c,\bm{E}_b,U^b)$, we have
	\beq\label{eq:cXb.cXc.entropyComp}
	\Ent(\cXb)
	\le \Ent(\cX^c) +
	\bigg( \Ent(\bm{E}_b) + \Ent(U^b \,|\, \bm{E}_b)\bigg)
	\stackrel{\eqref{e:DIVERSE.bc.ent.Estar}}{\le}
	\Ent(\cXc) + o_k(1)\,.
	\eeq
This verifies that the entropy of $\bPc$ is not much smaller than that of $\bPb$.

We finally estimate how the transformation from $\bPb$ to $\bPc$ affects the edge marginals. Recall that steps \hyperlink{i:DIVERSE.I1}{(I-1)} and 
\hyperlink{i:DIVERSE.I2}{(I-2)} can change $(\varsigma_e)^i$
from $\SPIN{red}$ to $\SPIN{white}$, and the probability of such a change is lower bounded by \eqref{e:DIVERSE.p.I.final.lbd}. Meanwhile, steps
\hyperlink{i:DIVERSE.II1}{(II-1)} and 
\hyperlink{i:DIVERSE.II2}{(II-2)} 
can change $(\varsigma_e)^i$ from $\SPIN{white}$ to $\SPIN{red}$,
and the probability of such a change is upper bounded by \eqref{e:DIVERSE.p.II.final.ubd}.
It follows that for $\bL\in\mathbb{D}\cap\mathbb{L}$,
	\begin{align}\nonumber
	&\bPc\bigg( (\sigma_e)^i=\red\,\bigg|\, \bL_e=\bL\bigg)
	\le
	\bPb\bigg( (\sigma_e)^i=\red\,\bigg|\, \bL_e=\bL\bigg)
	+\sum_{i=1,2}
	\f{p^{\textup{II},i}(e,\bL)- p^{\textup{I},i}(e,\bL)}
		{\pi_{\DD}(\bL\,|\,\bt_e)}
	\\
	\qquad&\le
	\bPb\bigg( (\sigma_e)^i=\red\,\bigg|\, \bL_e=\bL\bigg)
	+ O\bigg( \f{1}{2^{k(1+c_0)}}\bigg)
	-\Omega\bigg( \f{1}{k 2^{k(1+\EPSP/2)}}\bigg)
	\le
	\starpi_e(\red)
	-\Omega\bigg( \f{1}{k 2^{k(1+\EPSP/2)}}\bigg)\,,
	\label{eq:Xc.red.marginal.div}
	\end{align}
where the last step also used our earlier observation
\eqref{e:DIVERSE.ab.same.mgls} that $\bPb$ has the same edge marginals as $\bPa$. On the other hand, since the probability of
\hyperlink{i:DIVERSE.I1}{(I-1)} or
\hyperlink{i:DIVERSE.I2}{(I-2)} is upper bounded by \eqref{e:DIVERSE.p.I1.ubd.indiv}, we also have
	\beq\label{eq:Xc.red.marginal.div2}
	\bPc\bigg( (\sigma_e)^i=\red\,\bigg|\, \bL_e=\bL\bigg)
	\ge
	\bPb\bigg( (\sigma_e)^i=\red\,\bigg|\, \bL_e=\bL\bigg)
	-\sum_{i=1,2}
	\f{p^{\textup{I},i}(e,\bL)}{\pi_{\DD}(\bL\,|\,\bt_e)}
	\ge\starpi_e(\red)- 
	O\bigg(\f1{2^{k(1+\EPSP/2)}}\bigg)\eeq
for all $\bL\in\mathbb{D}\cap\mathbb{L}$. In the case $(\bL^b)\notin \mathbb{D}\cap\mathbb{L}$ we always set $(\cXc)_e=(\cXb)_e$,
so 
	\beq\label{eq:Xc.red.marginal.notdiv}
	\bPc\bigg( (\sigma_e)^i=\red\,\bigg|\, \bL_e=\bL\bigg)
	=
	\bPb\bigg( (\sigma_e)^i=\red\,\bigg|\, \bL_e=\bL\bigg)
	\eeq
for all $\bL\notin\mathbb{D}\cap\mathbb{L}$. Finally, since we did not yet change any spin with $(\varsigma_e)^i\in\set{\yel,\cya}$
or $\varsigma_e\in\set{\RYC}^2$, we have
	{\setlength{\jot}{0pt}\begin{align}
	\label{eq:Xc.col.marginal.zeta.mgl}
	\bPc((\varsigma_e)^i=\tau\,|\, \bL_e=\bL)
	&=(\bzeta_e)^i(\tau\,|\,\bL)
	\textup{ for all }\tau\in\set{\yel,\cya}, \\
	\bPc(\varsigma_e\in\set{\RYC}^2\,|\, \bL_e=\bL)
	&= \bPa(\varsigma_e\in\set{\RYC}^2\,|\, \bL_e=\bL)
	\le O(k^4/4^k)
	\label{eq:Xc.col.marginal}
	\end{align}}%
for all $\bL$, where the final bound uses \eqref{e:DIVERSE.Pa.RYCsq}.\smallskip

\noindent\bemph{\hypertarget{l:DIVERSE.Pd}{Part 5}. Construction of $\bPd$.} We now construct $\cXd$ from $\cXc$. We will leave the clause type unchanged, $(\bL^d)_e=(\bL^c)_e$ for all $e\in\delta v$. We change some edges from non-\SPIN{white} to \SPIN{white}, according to the following rules.
\begin{enumerate}[(A)]
\item If $\bL\not\in\mathbb{D}\cap\mathbb{L}$ and $(\varsigma^c)_e\ne\whi\whi$, then set $(\varsigma^d)_e=\whi\whi$, and
for $i=1,2$ set
	\[
	(\sigma^{d,i})_e
	=\begin{cases}
	\blu &\textup{if $(\sigma^{c,i})_e=\red$,}\\
	(\sigma^{c,i})_e &\textup{otherwise.}
	\end{cases}
	\]
\item If $\bL\in\mathbb{D}\cap\mathbb{L}$ then modify the spin on $e$ as follows:
\begin{enumerate}[(i)]
\item If $(\varsigma^c)_e=\red\red$ then set
$(\sigma^d,\varsigma^d)_e=(\blu\blu,\whi\whi)$.

\item If $(\varsigma^c)_e\in\set{\RYC}^2\setminus\set{\red\red}$ then set
$(\sigma^d)_e=(\sigma^c)_e$, and for $i=1,2$ set
	\[
	(\varsigma^{d,i})_e
	=\begin{cases}
	\red&\textup{if $(\sigma^{c,i})_e=\red$,}\\
	\whi&\textup{otherwise.}
	\end{cases}
	\]

\item If $(\varsigma^c)_e\in\set{\yel,\cya}\times\set{\whi}$ then set
$(\sigma^d)_e=(\sigma^c)_e$. For $\tau=(\varsigma^{c,1})_e\in\set{\yel,\cya}$ define the quantity
	\[\theta=
	(\theta_e)^1(\tau\,|\,\bL)
	\equiv\min\bigg\{1, \f{2^{-k(1+\EPSP/2)} }
		{(\bzeta_e)^1(\tau \,|\,\bL)}\bigg\}\,.
	\]
Set
 $(\varsigma^d)_e=\whi\whi$ with probability $\theta$. With the remaining probability $1-\theta$ leave $(\varsigma^d)_e$ unchanged.
If $(\varsigma^c)_e\in\set{\whi}\times\set{\yel,\cya}$ then perform the analogous procedure with the roles of the two copies switched.

\item Lastly, if $(\varsigma^c)_e=\whi\whi$ then set
$(X^d)_e=(\sigma^d,\varsigma^d)_e=(\sigma^c,\varsigma^c)_e=(X^c)_e$.
\end{enumerate}
\end{enumerate}
Note that property \eqref{e:DIVERSE.hat.E.nonempty}
of $\cXc$ ensures that the above transformation results in a valid configuration $\cXd$
with the same frozen spin, $x^c=x^d\in\set{\minus,\plus,\free}^2$. 
Since the transformation does not add any \SPIN{red} spins, for $\bL\in\mathbb{D}\cap\mathbb{L}$ we can bound
	\begin{align*}
	&\bPd\bigg(\varsigma_e=\red\whi\,\bigg|\,\bL_e=\bL\bigg)
	\le\bPc\bigg(\varsigma_e=\red\whi\,\bigg|\,\bL_e=\bL\bigg)
	\stackrel{\eqref{eq:Xc.red.marginal.div}}{\le} \starpi_e(\red)
	-\Omega\bigg( \f{1}{k 2^{k(1+\EPSP/2)}}\bigg)\\
	&\qquad=\zeta_e(\red\whi\,|\,\bL)
	+\zeta_e(\set{\red}\times\set{\RYC}\,|\,\bL)
	-\Omega\bigg( \f{1}{k 2^{k(1+\EPSP/2)}}\bigg)\\
	&\qquad\stackrel{\odot}{\le}
	\zeta_e(\red\whi\,|\,\bL)
	+O\bigg( \f1{2^{k(1+\EPSLIGHT)}}\bigg)
	-\Omega\bigg( \f{1}{k 2^{k(1+\EPSP/2)}}\bigg)
	\le \zeta_e(\red\whi\,|\,\bL)\,.
	\end{align*}
The inequality marked $\odot$ above follows from Lemma~\ref{l:lotsOfAsDiverse} and the assumption $\bL\in\mathbb{D}\cap\mathbb{L}$, with $\EPSLIGHT$ as in Definition~\ref{d:heavy}, and recalling that we took $\EPSP\le\EPSLIGHT$. For a lower bound, we note that for the case $\bL\in\mathbb{D}\cap\mathbb{L}$, the above transformation removes \SPIN{red} spins only in the case $(\sigma^c)_e=\red\red$. It follows that, for all $\bL\in\mathbb{D}\cap\mathbb{L}$, we have
	\begin{align*}
	&\bPd\bigg( (\varsigma_e)^1=\red
		\,\bigg|\, \bL_e=\bL\bigg)
	\ge \bPc\bigg( (\varsigma_e)^1=\red
		\,\bigg|\, \bL_e=\bL\bigg)
		-\bPc\bigg(
		\sigma_e=\red\red\,\bigg|\, \bL_e=\bL\bigg) \\
	&\stackrel{\eqref{e:DIVERSE.Pa.RYCsq}}{\ge} \bPc\bigg( (\varsigma_e)^1=\red
		\,\bigg|\, \bL_e=\bL\bigg)
		- O\bigg( \f{k^4}{4^k} \bigg)
	\stackrel{\eqref{eq:Xc.red.marginal.div2}}{\ge}
	 \starpi_e(\red) 
	 -O\bigg(\f1{2^{k(1+\EPSP/2)}}\bigg)
	 - O\bigg( \f{k^4}{4^k} \bigg) \\
	&\ge\zeta_e(\varsigma=\red\whi\,|\,\bL)
		- O\bigg(\f1{2^{k(1+\EPSP/2)}}\bigg)\,.
	\end{align*}
Similar bounds hold if we exchange the roles of the two copies, so we also have
	\[\zeta_e(\whi\red\,|\,\bL)
		- O\bigg(\f1{2^{k(1+\EPSP/2)}}\bigg)
	\le\bPd\bigg(\varsigma_e=\whi\red
		\,\bigg|\, \bL_e=\bL\bigg)
	\le \zeta_e(\whi\red\,|\,\bL)\]
for all $\bL\in\mathbb{D}\cap\mathbb{L}$. 
For $\bL\in\mathbb{D}\cap\mathbb{L}$ 
and $\tau\in\set{\yel,\cya}$ we also have
	\begin{align*}
	&\bPd\bigg(\varsigma_e=\tau\whi
		\,\bigg|\, \bL_e=\bL\bigg)
	\le
	\Big(1-(\theta_e)^1(\tau)\Big)
	\bPc\bigg((\varsigma_e)^1=\tau
		\,\bigg|\, \bL_e=\bL\bigg)
	\stackrel{\eqref{eq:Xc.col.marginal.zeta.mgl}}{=}
	\Big(1-(\theta_e)^1(\tau)\Big)
	(\bzeta_e)^i(\tau\,|\,\bL)\\
	&\qquad=\max\bigg\{0,
		(\bzeta_e)^i(\tau\,|\,\bL)-\f1{2^{k(1+\EPSP/2)}}
	\bigg\}
	=\max\bigg\{0,
		\bzeta_e(\tau\whi\,|\,\bL)
		+\bzeta_e(\tau\times\set{\RYC}\,|\,\bL)
		-\f1{2^{k(1+\EPSP/2)}}
	\bigg\}\\
	&\qquad\le \bzeta_e(\tau\whi\,|\,\bL)
		+O\bigg( \f1{2^{k(1+\EPSLIGHT)}}\bigg)
		-\f1{2^{k(1+\EPSP/2)}}
	\le \bzeta_e(\tau\whi\,|\,\bL)
		-\Theta\bigg(\f1{2^{k(1+\EPSP/2)}}\bigg)\,,
	\end{align*}	
where the transition to the last line is by another application of Lemma~\ref{l:lotsOfAsDiverse}, again using that $\bL\in\mathbb{D}\cap\mathbb{L}$. For a lower bound on the same quantity, we have
for all $\bL\in\mathbb{D}\cap\mathbb{L}$ and $\tau\in\set{\yel,\cya}$ that
	\begin{align*}
	&\bPd\bigg(\varsigma_e=\tau\whi
		\,\bigg|\, \bL_e=\bL\bigg)
	\ge\Big(1-(\theta_e)^1(\tau)\Big)
	\bigg\{
	\bPc\bigg( (\varsigma_e)^1=\tau
		\,\bigg|\, \bL_e=\bL\bigg)
	-\bPc\bigg( \varsigma_e\in\set{\tau}\times\set{\RYC}
		\,\bigg|\, \bL_e=\bL\bigg)\bigg\}\\
	&\qquad\stackrel{\eqref{eq:Xc.col.marginal}}{\ge}
		\Big(1-(\theta_e)^1(\tau)\Big)
	\bPc\bigg( (\varsigma_e)^1=\tau
		\,\bigg|\, \bL_e=\bL\bigg)
		-O\bigg(\f{k^4}{4^k}\bigg)\\
	&\qquad\stackrel{\eqref{eq:Xc.col.marginal.zeta.mgl}}{=}
	\max\bigg\{0,
		\bzeta_e(\tau\whi\,|\,\bL)
		+\bzeta_e(\tau\times\set{\RYC}\,|\,\bL)
		-\f1{2^{k(1+\EPSP/2)}}
	\bigg\}-O\bigg(\f{k^4}{4^k}\bigg)
	\ge \bzeta_e(\tau\whi\,|\,\bL)
	-\Theta\bigg(\f1{2^{k(1+\EPSP/2)}}\bigg)\,.
	\end{align*}
Finally, the above procedure never yields $(\varsigma^d)_e\in\set{\RYC}^2$, so
for all $\bL\in\mathbb{D}\cap\mathbb{L}$
and all $\varsigma\in\set{\RYC}^2$ we have
	\[
	\bzeta_e(\varsigma\,|\,\bL)
	\ge 0 = \bPd\bigg(\varsigma_e=\varsigma
		\,\bigg|\,\bL_e=\bL\bigg)
	\ge \bzeta_e(\varsigma\,|\,\bL)
		-\Theta\bigg(\f1{2^{k(1+\EPSP/2)}}\bigg)\,,
	\]
where the last inequality again uses Lemma~\ref{l:lotsOfAsDiverse}. In summary, in the above we have shown
	\beq\label{e:DIVERSE.cd.zeta.mgl}
	\bzeta_e(\varsigma\,|\,\bL)
	\ge \bPd\bigg(\varsigma_e=\varsigma
		\,\bigg|\,\bL_e=\bL\bigg)
	\ge \bzeta_e(\varsigma\,|\,\bL)
		-\Theta\bigg(\f1{2^{k(1+\EPSP/2)}}\bigg)
		\eeq
for all $\bL\in\mathbb{D}\cap\mathbb{L}$
and all $\varsigma\in\set{\RYC,\whi}^2\setminus\set{\whi\whi}$. On the other hand, for $\bL\notin\mathbb{D}\cap\mathbb{L}$ we have
	\beq\label{e:DIVERSE.cd.zeta.mgl.notDL}
	\bzeta_e(\varsigma\,|\,\bL)
	\ge0=
	\bPd\bigg(\varsigma_e=\varsigma
		\,\bigg|\,\bL_e=\bL\bigg)
	\ge \bzeta_e(\varsigma\,|\,\bL)
	- O\bigg( \f{k^2}{2^k}\bigg)
	\eeq
for all $\varsigma\in\set{\RYC,\whi}^2\setminus\set{\whi\whi}$, where the last inequality is by Lemma~\ref{l:lotsOfAs}.

We now turn to estimating the entropy of $\bPd$. To this end, let $\bm{E}_c$ denote the subset of edges $e\in\delta v$ for which $(\cXd)_e\ne(\cXc)_e$, or equivalently $(X^d)_e\ne(X^c)_e$. Note that if 
$(\varsigma^{c,i})_e\ne\whi$
and $(\varsigma^{d,i})_e\ne\whi$,
then $(\varsigma^{c,i})_e=(\varsigma^{d,i})_e$
so $e\notin\bm{E}_c$.
If $e\in\bm{E}_c$ then
the above procedure must transform
$(\varsigma^{c,i})_e\ne\whi$ into $(\varsigma^{d,i})_e=\whi$ for at least one index $i=1,2$, so we can bound
	\begin{align*}
	&\P(e\in \bm{E}_c\,\bigg|\,\bL_e=\bL)
	\le \sum_{i=1,2}
	\bigg\{ \bPc\Big( (\varsigma_e)^i\ne\whi
		\,\Big|\,\bL_e=\bL\Big)
	-\bPd\Big( (\varsigma_e)^i\ne\whi
		\,\Big|\,\bL_e=\bL\Big)\bigg\}\\
	&\qquad\le
	\sum_{i=1,2}\bigg\{
	(\bzeta_e)^i\Big(\set{\RYC}
		\,\Big|\,\bL\Big)
	-\bPd\Big( (\varsigma_e)^i\ne\whi
		\,\Big|\,\bL\Big)\bigg\}
	\le O(1)\sum_{\varsigma\ne\whi\whi}
		\bigg|
		\bzeta_e(\varsigma\,|\,\bL)
		- \bPd\Big(\varsigma_e=\varsigma
			\,\Big|\,\bL_e=\bL\Big)
		\bigg|\,,
	\end{align*}
where the transition to the second line follows by combining \eqref{eq:Xc.red.marginal.div} and \eqref{eq:Xc.col.marginal.zeta.mgl}. Applying
\eqref{e:DIVERSE.cd.zeta.mgl} and \eqref{e:DIVERSE.cd.zeta.mgl.notDL} gives
	\begin{align}\nonumber
	\E(|\bm{E}_c|)
	&\le
	O(1)\sum_{e\in\delta v}
		\sum_{\varsigma\ne\whi\whi}
		\bigg\{
		\f{\bPc(\bL_e\in\mathbb{D}\cap\mathbb{L})}
			{2^{k(1+\EPSP/2)}}
		+\f{\bPc(\bL_e\notin\mathbb{D}\cap\mathbb{L})}
			{2^k/k^2}\bigg\}\\
	&\le O\bigg(\f{k}{2^{k\EPSP/2}}
	+ \f{\notDiverse(v)+\notLight(v)}
		{2^k/k^2}\bigg)
	\stackrel{\eqref{e:assumption.notD.notL}}{\le}
	O\bigg(\f{k^2}{2^{k\EPSP/2}}\bigg)\,,
	\label{e:DIVERSE.Ec.bound}\end{align}
similarly to the estimate
\eqref{e:DIVERSE.expected.bc.change} from the previous transformation. It follows by the same argument as for \eqref{eq:cXb.cXc.entropyComp} that
	\beq\label{eq:cXc.cXd.entropyComp}
	\Ent(\bPc) 
	=\Ent(\cXc) 
	\le \Ent(\cXd) + o_k(1)
	= \Ent(\bPd) + o_k(1).
	\eeq
This concludes our analysis of the measure $\bPd$.\smallskip

\noindent\bemph{\hypertarget{l:DIVERSE.Pe}{Part 6}. Construction of $\bPe$.}
In the final step of the proof we construct $\cXe$ from $\cXd$, where we have chosen the notation $\cXe$ rather than $\cX^e$ to avoid confusion with edge labels $e\in\delta v$. For $\sigma\in\set{\yel,\blu}^2$ we define subsets $A^\sigma$ of $
\set{\RYGB}^2\times\set{\RYC,\whi}^2$ as follows:
	{\setlength{\jot}{0pt}\begin{align*}
	A^{\blu\blu}&=\set{(\red\red,\red\red),
		(\blu\red,\whi\red),(\blu\red,\cya\red),
		(\red\blu,\red\whi),(\red\blu,\red\cya),
		(\blu\blu,\whi\cya),(\blu\blu,\cya\whi),
		(\blu\blu,\cya\cya)}\,,\\
	A^{\blu\yel}&=\set{
		(\red\yel,\red\yel),
		(\blu\yel,\whi\yel),
		(\blu\yel,\cya\yel)
		}\,,\\
	A^{\yel\blu}&=\set{
		(\yel\red,\yel\red),
		(\yel\blu,\yel\whi),
		(\yel\blu,\yel\cya)
		}\,,\\
A^{\yel\yel}&=\{(\yel\yel,\yel\yel)\}\,.
\end{align*}}%
Note that for each $\sigma\in\set{\yel,\blu}^2$ and every $(\sigma',\varsigma')\in A^\sigma$ we have the compatibility relation $\sigma'\sim\varsigma'$; and the spins $\sigma,\sigma'$ correspond to the same frozen spin $x\in\set{\minus,\plus}^2$. Moreover, for each $\varsigma\in\set{\RYC,\whi}^2\setminus\set{\whi\whi}$ 
there is exactly one spin $\sigma\equiv\sigma^f(\varsigma)\in\set{\yel,\blu}^2$ such that 
$(\sigma',\varsigma)\in A^\sigma$ for some $\sigma'$.
We then define distributions $\xi^\sigma(\cdot\,|\,\bL)$ over pairs $X\equiv (\sigma,\varsigma)\in\set{\RYGB}^2\times\set{\RYC,\whi}^2$ by setting
	\beq\label{e:DIVERSE.xi.defn}
	\xi^\sigma((\sigma,\varsigma)\,|\,\bL)
	=\f{\bzeta_e(\varsigma\,|\,\bL)
	-\bPd(\varsigma_e = \varsigma \,|\, \bL_e=\bL)}{\bPd(\sigma_e = \sigma, \varsigma_e= \whi\whi
		\,|\, \bL_e=\bL )}
	\eeq
for all $(\sigma,\varsigma)\in A^\sigma$, and assigning the remaining probability to the event $\set{(\sigma,\varsigma)=(\sigma,\whi\whi)}$. To see that these distributions are well-defined, note that on the right-hand side of \eqref{e:DIVERSE.xi.defn}, the denominator is $\Theta(1)$, while \eqref{e:DIVERSE.cd.zeta.mgl} and \eqref{e:DIVERSE.cd.zeta.mgl.notDL} together imply that the numerator is nonnegative and small. It follows that
	\[
	\xi^\sigma((\sigma,\whi\whi)\,|\,\bL)
	= 1-\sum_{X'\in A^\sigma}
		\xi^\sigma(X'\,|\,\bL)
	\in[0,1]\,,
	\]
so $\xi^\sigma(\cdot\,|\,\bL)$ is a valid probability measure over $A^\sigma \cup \set{(\sigma,\whi\whi)}$. In the final transformation, given $\cXd$, for each $e\in\delta v$ we set $(\bL^f)_e=(\bL^d)_e$. If $(\sigma^d)_e\in\set{\yel,\blu}^2$ and $(\varsigma^d)_e=\whi\whi$, then we let $(X^f)_e$ be a sample from $\xi^\sigma(\cdot\,|\,(\bL^d)_e)$. In all other cases we let $(X^f)_e=(X^d)_e$. It is straightforward to verify that this results in a valid configuration $\cXe$ with the same frozen spin,
$x^f=x^d\in\set{\minus,\plus,\free}^2$. For each $\varsigma\in\set{\RYC,\whi}^2\setminus\set{\whi\whi}$,
	\begin{align}\nonumber
	&\bPe\Big(\varsigma_e=\varsigma\,\Big|\,\bL_e=\bL\Big)
	= \bPd\Big(\varsigma_e=\varsigma\,\Big|\,\bL_e=\bL\Big)
	+ \sum_{\sigma: \sigma=\sigma^f(\varsigma)}
	\bPd\Big(
	\sigma_e=\sigma,
	\varsigma_e=\whi\whi
	\,\Big|\,\bL_e=\bL\Big)
	\xi^\sigma( (\sigma,\varsigma) \,|\,\bL)\\
	&\qquad\stackrel{\eqref{e:DIVERSE.xi.defn}}{=}
	\bPd\Big(\varsigma_e=\varsigma\,\Big|\,\bL_e=\bL\Big)
	+ \bigg\{ \bzeta_e(\varsigma\,|\,\bL)
	-\bPd\Big(\varsigma_e=\varsigma\,\Big|\,\bL_e=\bL\Big)\bigg\}
	=\bzeta_e(\varsigma\,|\,\bL)
	\,,\label{e:DIVERSE.final.msr.sat.varsigma}
	\end{align}
so $\bPe$ satisfies condition
\eqref{e:D.L.diverse.varsigma}. On the other hand, since $x^a=x^f\in\set{\minus,\plus,\free}^2$, we have
	\[
	\bPe\Big( x^i=z\,\Big|\,\bL_e=\bL\Big)
	=\bPa\Big( x^i=z\,\Big|\,\bL_e=\bL\Big)
	\stackrel{\eqref{eq:productForm}}{=} \begin{cases}
	\starpi_e(\set{\red,\blu})
	&\textup{if $z\in\set{\minus,\plus}$
		and $e\in\delta v(z)$,}\\
	\starpi_e(\yel)
	&\textup{if $z\in\set{\minus,\plus}$
		and $e\in\delta v(\minus z)$,}\\
	\starpi_e(\grn) 
	&\textup{if $z=\free$.}
	\end{cases}
	\]
Moreover, since $(\varsigma_e)^i=\red$ if and only if $(\sigma_e)^i=\red$, it follows that
	\[
	\bPe\Big( (\sigma_e)^i=\red\,\Big|\,\bL_e=\bL\Big)
	\stackrel{\eqref{e:DIVERSE.final.msr.sat.varsigma}}{=}
	(\bzeta_e)^i(\red\,|\,\bL)
	= \starpi_e(\red)\,.
	\]
It follows from the last two displays combined that $\bPe$ satisfies condition~\eqref{e:D.L.diverse.judicious}.
Lastly, since $(\bL^a)_e=(\bL^f)_e$ for all $e$, 
and $\bPa$ satisfied condition \eqref{e:D.L.diverse.clause} by construction,
we can conclude that $\bPe$ also satisfies 
condition~\eqref{e:D.L.diverse.clause}.

We now estimate the entropy of $\bPe$.
Let $\bm{E}_d$ denote the subset of edges $e\in\delta v$ for which $(\cXd)_e\ne(\cXe)_d$. For any edge in $e\in \bm{E}_d$ we must have $(\varsigma^d)_e=\whi\whi$ and $(\varsigma^f)_e\ne\whi\whi$. As a result we can bound
	\begin{align*}
	\E(|\bm{E}_d|)
	&\le \sum_{e\in \delta v}
	\bigg(
	\bPe(\varsigma_e\ne\whi\whi)
	- \bPd(\varsigma_e\ne\whi\whi)
	\bigg)
	= \sum_{e\in \delta v}
	\sum_{\varsigma\ne\whi\whi}\bigg(
	 \bzeta_e(\varsigma)
	 - \bPd(\varsigma_e=\varsigma)
	 \bigg) \le 
	 O\bigg( \f{k^2}{2^{k \EPSP/2}}\bigg)\,,
\end{align*}
using the same reasoning as for \eqref{e:DIVERSE.Ec.bound}. It then follows by the argument of~\eqref{eq:cXb.cXc.entropyComp} that
	\beq\label{eq:cXd.cXe.entropyComp}
	\Ent(\bPd)\le \Ent(\bPe) + o_k(1)\,.\eeq
Combining equations~\eqref{eq:A.to.B.Ent.Comp}, \eqref{eq:cXb.cXc.entropyComp}, \eqref{eq:cXc.cXd.entropyComp},\eqref{eq:cXd.cXe.entropyComp},
and \eqref{eq:a.opt.EntComp} gives
	\[
	\Ent(\bPe)
	\ge \Ent(\bPa) - o_k(1)
	\stackrel{\eqref{eq:a.opt.EntComp}}{\ge}
	\Ent(\bP) + \f1{10^6} - o_k(1)\,.
	\]
This contradicts the assumption that $\bP$ is the maximal-entropy measure satisfying constraints
\eqref{e:D.L.diverse.clause}--\eqref{e:D.L.diverse.varsigma}, thereby concluding the proof of the proposition.
\end{proof}
\end{ppn}

\subsection{Conclusion of \textit{a~priori} estimates}\label{ss:conclusion.apriori}

In this subsection we complete the proof of Proposition~\ref{p:apriori}. We first prove the expansion result, Lemma~\ref{l:expand.on.types}, which was stated in \S\ref{ss:hess}.

\begin{proof}[Proof of Lemma~\ref{l:expand.on.types}]
We first show that in the original $\ksat$ graph
$\GG'=(V',F',E')\sim\P=\P_{n,m}$, \bemph{every} subset of variables $S\subseteq V'$ satisfies the bound
	\beq\label{e:expansion.bd.orig}
	|F_\bullet(S)|\le \begin{cases}
		n 2^{29k/30}/k & 
		\textup{if $0 \le |S|/n \le 4/5$,}\\
		|S| &\textup{if $c_1/2 \le |S|/n \le 1/4$,}
		\end{cases}\eeq
where $c_1$ is as in Proposition~\ref{p:posfrac}.
(Note that \eqref{e:expansion.bd} refers to the number $|V|$ of variables in the processed graph, while in \eqref{e:expansion.bd.orig} and throughout this proof we use $n\equiv |V'|\ge|V|$ to refer to the number of variables in the original graph. Thus in \eqref{e:expansion.bd.orig} we crudely divided the first bound by a factor $k$, which we will use below to account for the discrepancy in the number of variables in $\GG'$ versus in $\GG$.) Towards the proof of \eqref{e:expansion.bd.orig}, recall that $F_\bullet(S)$ denotes the subset of clauses $a\in F$ with $|S\cap \pd a|\ge 9k/10$. Note that for any fixed subset $S\subseteq V'$ of size $|S|=ns$, the probability that a particular clause $a\in F'$ belongs to $F_\bullet(S)$ is given by
	\[
	p_s
	= \P\bigg( \textup{Bin}(k,s) \ge \f{9k}{10}\bigg)
	\le
	\exp \Bigg\{
		-k\Ent\bigg(\f{9}{10}\,\bigg|\,s\bigg)
		\Bigg\}
	\le
	 \begin{cases}
	2^{-k/20}&\textup{for all $0\le s\le 4/5$,}\\
	s^{2k/3}&\textup{for all $0\le s\le 1/4$,}
	\end{cases}\]
where the first inequality is from the Chernoff bound and the second is by direct comparison.
If $\E$ denotes expectation over the law $\P$ of $\GG'$, then another application of the Chernoff bound gives
	\begin{align*}
	E_n(s,\gamma)
	&\equiv
	\E \Bigg| \bigg\{
	S : |S|=ns \textup{ and }
		|F_\bullet(S)| \ge n\gamma\bigg\}\Bigg|
	\le \binom{n}{ns}
	\P\bigg( \textup{Bin}(m,p_s) \ge n\gamma\bigg)\\
	&\le
	\exp\Bigg\{
	n\Bigg[
	\Ent(s) - \alpha \Ent\bigg( \f{\gamma}{\alpha}
		\,\bigg|\, p_s\bigg)\Bigg]
	\Bigg\}\,.
	\end{align*}
For $0\le s\le 4/5$, taking $\gamma= 2^{29k/30}/k$ gives
	\[
	\alpha \Ent\bigg(
		\f{\gamma}{\alpha}\,\bigg|\, p_s\bigg)
	= \gamma
	\log \f{\gamma/\alpha}{p_s}
	+ (\alpha-\gamma)
	\log \f{1-\gamma/\alpha}{1-p_s}
	\ge
	\Omega(\gamma k)
	- O(\gamma + \alpha p_s)
	\ge \Omega(\gamma k)\,,
	\]
so we see that $E_n(s,\gamma)$ is exponentially small in $n$. For $c_1/2\le s\le 1/4$, taking $\gamma=s$ gives
	\[
	\alpha \Ent\bigg(
		\f{\gamma}{\alpha}\,\bigg|\, p_s\bigg)
	\ge s \log \f{s/\alpha}{s^{2k/3}}
	- O(s + \alpha p_s)
	\ge \Omega\bigg( k s\log \f1s\bigg)
	\ge \Omega\bigg( k\Ent(s)\bigg)\,,
	\]
so again $E_n(s,\gamma)$ is exponentially small in $n$. It follows that for the original $\ksat$ instance, we have
	\beq\label{e:expo.bound.on.expansion}
	\P\bigg(
	\textup{any subset $S'\subseteq V'$
	violates \eqref{e:expansion.bd.orig}}
	\bigg)
	\le
	\exp\Bigg\{-n\Omega\bigg(
		k\Omega\bigg(\Ent\bigg(\f{c_1}{2}
		\bigg)\bigg)\Bigg\}
	\le \exp\{-nc'\}\,,\eeq
where $c'$ depends only on $k$ and $c_1$ (and $c_1$ in turn depends only on $k,\alpha,R$). If in the processed $\ksat$ instance $\GG=(V,F,E)$ we have any subset $S\subseteq V$ violating \eqref{e:expansion.bd},
then in the original instance
$\GG'=(V',F',E')$ it must be the case that either some subset
$S'\subseteq V'$ violates \eqref{e:expansion.bd.orig},
or $|V| \le |V'|/k$. It follows by combining with 
\eqref{e:expo.bound.on.expansion} and 
Proposition~\ref{p:small.fraction.removed.in.processing} that 
	\begin{align*}
	&\E\bigg[\P_{\DD}\Big(
	\textup{any subset $S\subseteq V$
	violates \eqref{e:expansion.bd}}
	\Big)\bigg] \\
	&\le 
	\P\bigg( |V| \le \f{|V'|}{k}\bigg)
	+\P\bigg(
	\textup{any subset $S'\subseteq V'$
	violates \eqref{e:expansion.bd.orig}}
	\bigg)
	\le o_n(1)\,,
	\end{align*}
and the claimed result follows by Markov's inequality.
\end{proof}

\begin{ppn}[used only in proof of Proposition~\ref{p:apriori}]
\label{p:noNonDiverse} On the event that $\DD$ expands on type-subsets (Definition~\ref{d:expansion}), all non-defective variables are diverse (Definition~\ref{d:diverse}); and all strongly 
non-defective clauses are light (Definition~\ref{d:heavy}).

\begin{proof} Since we have restricted to $\omega\in\bm{I}_0$ (as has been the case throughout this section), the total number of variables $v$ with frozen spin 
$x_v\in\set{\plus\plus,\minus\minus}$ is close to $n/2$, in any pair frozen configuration that is consistent with $\omega$. Let $S'$ denote the set of all non-diverse variables: then
	\beq\label{e:diverse.Markov}
	|S'|
	= \sum_{v\in V}
	\mathbf{1}\bigg\{
	\pi_v(\set{\plus\plus,\minus\minus}) > \f34 
	\bigg\}
	\le \f43
	\sum_{v\in V}
	\pi_v(\set{\plus\plus,\minus\minus})
	\le \f{3n}{4}
	\eeq
(where the last inequality uses the above observation on the number of variables with $x_v\in\set{\plus\plus,\minus\minus}$).

Let $\notDiverse\subseteq F$ denote the set of non-diverse clauses; it follows from Definitions~\ref{d:diverse}~and~\ref{d:expansion} that $\notDiverse\subseteq F_\bullet(S')$. Now let $S\subseteq S'$ denote the subset of non-diverse variables that are non-defective (Definition~\ref{d:j.defective}), so we see from the above that
$|S| \le |S'| \le 3n/4$. Recall from Remark~\ref{r:defective.clauses}
that a clause is termed strongly non-defective if it neighbors only non-defective variables. Let $N\subseteq\notDiverse$ denote the set of non-diverse clauses that are strongly non-defective; it follows that $N\subseteq F_\bullet(S)$. Then, on the event that $\DD$ expands on type-subsets, it follows from Definition~\ref{d:expansion} that
	\beq\label{e:diverseClauseBound}
	|N|
	\le \left\{\hspace{-6pt}\begin{array}{rl}
		n 2^{29k/30} & 
		\text{if }|S|/n \le 4/5\,.\\
		|S| &\text{if } |S|/n \le 1/16\,.
		\end{array}\right.\eeq
Next let $\notLight\subseteq F$ denote the set of all heavy clauses, and let $H\subseteq\notLight$ denote the subset of heavy clauses that are strongly non-defective. By Proposition~\ref{p:DIVERSE.VAR}, if $v$ is non-defective but not diverse, then
	\[
	\f{2^k}{2^{k\EPSP}}<
	\max\Big\{\notDiverse(v),
		\notLight(v)\Big\}
	\le \notDiverse(v)
		+ \notLight(v)\,,
	\]
where we can take $\EPSP$ to be much smaller than all the other constants $\ep$ in this section. Summing the above over all $v\in S$,
and recalling the definitions
of $\notDiverse(v)$ and $\notLight(v)$
from Definitions~\ref{d:diverse}~and~\ref{d:heavy}, we obtain
	\beq\label{e:noNonDiverse.S.leq.N.H}
	\f{2^k |S|}{2^{k\EPSP}} 
	< \sum_{v\in S}
	\sum_{e\in\delta v}
	\sum_{\bL}
	\pi_{\DD}(\bL\,|\,\bt_e)
	\bigg(
	\Ind{\bL\in\notDiverse
		\cup \notLight}\bigg)
	\le \Big(|N| + |H|\Big)k + |S|\,,
	\eeq
where the right-hand side of \eqref{e:noNonDiverse.S.leq.N.H} is obtained as follows: the contribution from the case of strongly non-defective clause types $\bL$ is bounded by the term $(|N|+|H|)k$, where the factor $k$ accounts for the fact that one clause can neighbor up to $k$ variables. As for the case where $\bL$ fails to be strongly non-defective, we recall from Remark~\ref{r:defective.clauses} that each $v\in S$ neighbors at most one such clause, so the contribution from this case is bounded by the term $|S|$. 

Now note that \eqref{e:diverseClauseBound} bounds $|N|$ in terms of $|S|$, while \eqref{e:noNonDiverse.S.leq.N.H} bounds $|S|$ in terms of $|N|$ and $|H|$. We will close the loop by bounding $|H|$ in terms of $|N|$. In fact we will bound $|E_H|$ where $E_H$ is a subset of edges defined as follows: for any edge $e=(av)$ where $\bL_a=\bL$ and $j=j(\bt_e)$, write $\omega_e\equiv\omega_{\bL,j}$. Then let
	\[
	E_H
	\equiv\bigg\{
	e=(av): \textup{$a$ is a strongly non-defective clause
		and }\omega_e(\red\red)
		\ge \f1{2^{k(1+\EPSLIGHT)}}
	\bigg\}\,.
	\]
Note that a clause belongs to $H$ if and only if it lies incident to $E_H$, so $|H|\le |E_H|$. We will bound
	\[
	|E_H|
	\le \sum_{i=0}^2 |E_H(U_i)|
	\]
where $E_H(U_i)$ denotes the subset of edges in $E_H$ that are incident to $U_i$, and $(U_0,U_1,U_2)$ gives a partition of the set of all non-defective variables: 
	{\setlength{\jot}{0pt}\begin{align*}
	U_0 &\equiv
	\set{\text{non-defective variables $v$}:
		\notDiverse(v)
		\le 2^{k(1-\EPSP)},
		\notLight(v)
		\le 2^{k(1-\EPSP)}
		}\,,\\
	U_1 &\equiv
	\set{\text{non-defective variables $v$}:
		\notDiverse(v)
		\le 2^{k(1-\EPSP)},
		\notLight(v)
		> 2^{k(1-\EPSP)}
		}\,,\\
	U_2 &\equiv
	\set{\text{non-defective variables $v$}:
		\notDiverse(v)
		> 2^{k(1-\EPSP)}	}\,.
	\end{align*}}%
It follows from Proposition~\ref{p:DIVERSE.VAR} that
if $v\in U_0$ then $v$ is diverse, and moreover that $\pi_v(x)$ is close to $1/4$ for each $x\in\set{\minus,\plus}^2$. Thus each $v\in U_0$ satisfies the conditions of Lemma~\ref{l:large11Case} --- in particular, Lemma~\ref{l:large11Case} has a condition $\pi_v(\plus\plus)\ge 2^{-k/16}$ which is certainly guaranteed if $\pi_v(\plus\plus)$ is close to $1/4$.
It follows by Lemma~\ref{l:large11Case} that an edge $e=(av)$ with $v\in U_0$ cannot belong to $E_H$ unless the clause $a$ is non-diverse, which means $a\in N$. Therefore
	\[
	|E_H(U_0)| \le \bigg|\Big\{
		(av) \in E_H : a\in N\Big\}\bigg|
	\le |N|k\,.
	\]
Next, similarly to \eqref{e:noNonDiverse.S.leq.N.H} we have the bound
	\beq\label{e:EPSAP.U.one.bound}
	\f{2^k |U_1|}{2^{k\EPSP}}
	<
	\sum_{v\in U_1} \notLight(v)
	\le |H|k + |U_1| \,.
	\eeq
For $v\in U_1$, applying Corollary~\ref{c:RRFreq} gives
	\beq\label{e:EPSAP.avg.heavy.bound}
	\sum_{e\in\delta v}
	\sum_{\bL\in\mathbb{D}}
	\pi(\bL|\bt_e)
	\mathbf{1}\Bigg\{ \omega_{\bL,j}(\red\red)
	\ge \f1{2^{k(1+\EPSLIGHT)} }\Bigg\}
	\le \f{{k^2}
		2^{k(1+\EPSLIGHT)}}{2^{k\EPSDIV}}
	\le \f{2^k}{2^{k\EPSLIGHT}}\,,\eeq
where the last bound holds because we took $\EPSLIGHT$ to be much smaller than $\EPSDIV$. Combining with the preceding upper bound on $|U_1|$ gives
	\begin{align*}
	|E_H(U_1)|
	&\le
	\bigg|\Big\{ (av) : a\in N\Big\}\bigg|
	+\bigg|\Big\{ (av)\in E_H : v\in U_1,
		a\in\mathbb{D}\Big\}\bigg|\\
	&\stackrel{\eqref{e:EPSAP.avg.heavy.bound}}{\le}
	|N|k
	+|U_1| \f{2^k}{2^{k\EPSLIGHT}}
	\stackrel{\eqref{e:EPSAP.U.one.bound}}{\le}
	|N|k
	+|H| \f{2^{k\EPSP}}{2^{k\EPSLIGHT}} k
		\bigg(1+o_k(1)\bigg)\,.
	\end{align*}
Finally, similarly to \eqref{e:noNonDiverse.S.leq.N.H} and \eqref{e:EPSAP.U.one.bound} we have
	\beq\label{e:EPSAP.U.two.bound}
	\f{2^k|U_2| }{2^{k\EPSP}}
	< \sum_{v\in U_2} \notDiverse(v)
	\le |N|k + |U_2|\,,\eeq
from which it follows that
	\[
	|E_H(U_2)|
	\le \bigg|\Big\{(av) : v\in U_2\Big\}\bigg|
	\le |U_2| 2^k k^2
	\stackrel{\eqref{e:EPSAP.U.two.bound}}{\le}
	|N| 2^{k\EPSP}k^3 
	\bigg(1+o_k(1)\bigg)\,.\]
Combining the above bounds gives
	\[|H| \le \sum_{i=0}^2 |E_H(U_i)| 
	\le |N| 2^{k\EPSP} k^4
		+ |H| \f{2^{k\EPSP}}{2^{k\EPSLIGHT}} k^2
	\le
	|N| 2^{k\EPSP}k^4
	+ \f{|H|}{2^{k\EPSLIGHT/2}}
	\,,\]
where the last bound holds since we took $\EPSP$ to be much smaller than $\EPSLIGHT$, as discussed above. Rearranging gives
$|H| \le O(1) |N| 2^{k\EPSP}k^4$, and combining
with \eqref{e:noNonDiverse.S.leq.N.H} gives
	\[
	|S|
	\stackrel{\eqref{e:noNonDiverse.S.leq.N.H}}{\le}
	O(1)
	\Big(|N|+|H|\Big) \f{2^{k\EPSP}}{2^k} k
	\le \f{|N| k^{O(1)}}{2^{k(1-2\EPSP)}}\,.
	\]
Together with the expansion bound 
\eqref{e:diverseClauseBound}
(which can be applied, since we saw earlier that $|S| \le 3n/4$), we see that the only possibility is for $|S|=|N|=0$. This implies $|H|=0$ and concludes the proof.
\end{proof}
\end{ppn}

\begin{proof}[Proof of Proposition~\ref{p:apriori}]
If $\DD$ expands on type-subsets, then Proposition~\ref{p:noNonDiverse} yields that all non-defective variables are diverse, and all strongly non-defective clauses are light. Moreover, since a non-defective clause can neighbor at most one defective variable, it follows that every non-defective clause is diverse in the sense of Definition~\ref{d:diverse}. Thus, if a variable $v$ is strongly non-defective, then all the clauses around it must be both diverse and light, meaning $\notDiverse(v)=\notLight(v)=0$. Proposition~\ref{p:balancedeVertex} then gives the desired estimate \eqref{e:balancedeVertex} (by taking $\ZETA=\EPSLIGHT$) for $\omega_{\bL,j}$ whenever $\bL(j)$ is a strongly non-defective edge type.
\end{proof}

\section{Monotonicity of the $\onersb$ free energy}
\label{s:monotonicity}

In this section we prove Proposition~\ref{p:phi}, which says that the $\onersb$ free energy is strictly decreasing with respect to $\alpha$ in the interval \eqref{e:alpha.regime}. Throughout this section we let
$\albd \le \dot{\alpha}\le\ddot{\alpha}
\le\aubd$.

\begin{dfn}[monotone coupling of trees] 
\label{d:monotone.coupling}
Recall from
Definition~\ref{d:vtoc.PGW}
that $\PGW\equiv\PGW(\alpha)$ denotes the law of the variable-to-clause tree $\tree_{\vrt\crt}$. We now define a pair of such trees as follows. First let $\ddot{\tree}_{\vrt\crt}\sim\PGW(\ddot{\alpha})$. 
Then delete each clause in
$\ddot{\tree}_{\vrt\crt}\sim\PGW(\ddot{\alpha}) \setminus \set{\crt}$ independently with probability $1-\dot{\alpha}/\ddot{\alpha}$, and let
$\dot{\tree}_{\vrt\crt}$ be the connected component containing $\vrt$, so that
$\dot{\tree}_{\vrt\crt}\sim\PGW(\dot{\alpha})$. Let $\PGW(\dot{\alpha},\ddot{\alpha})$ denote the law of the pair 
$(\dot{\tree}_{\vrt\crt},\ddot{\tree}_{\vrt\crt})$. This is the \bemph{monotone coupling} between
the measures $\PGW(\dot{\alpha})$ and $\PGW(\ddot{\alpha})$. Recalling Definition~\ref{d:coupled.seq.of.messages}, let
	\[
	(\dot{\bmeta}^\ell,
	\ddot{\bmeta}^\ell)
	\equiv \bigg(
	\FF_\ell(\dot{\tree}_{\vrt\crt} ),
	\FF_\ell(\ddot{\tree}_{\vrt\crt} )
	\bigg)
	\,,
	\]
for all $\ell\ge0$. Recall that $\dot{\bmeta}^\ell$ and $\ddot{\bmeta}^\ell$ are (random) probability measures over $\set{\minus,\plus,\free}$; similarly as before we will denote
$\dot{\eta}^\ell\equiv\dot{\bmeta}^\ell(\minus)$
and
$\ddot{\eta}^\ell\equiv\ddot{\bmeta}^\ell(\minus)$.
In the notation of Proposition~\ref{p:fp}, the marginal law of $\dot{\eta}^\ell$ is the measure $\mu^\ell(\dot{\alpha})$, while that of $\ddot{\eta}^\ell$ is $\mu^\ell(\ddot{\alpha})$.
We hereafter write $\mu^\ell(\dot{\alpha},\ddot{\alpha})$ for the law of the pair $(\dot{\eta}^\ell,\ddot{\eta}^\ell)$. 
\end{dfn}

\begin{lem}\label{lem-Q-monotone-coupling-prelim}
Let $0\le \ddot{\alpha}-\dot{\alpha} \le \exp(-2^k)$, and let $(\dot{\eta}^\ell,\ddot{\eta}^\ell)\sim \mu^\ell(\dot{\alpha},\ddot{\alpha})$
as described in Definition~\ref{d:monotone.coupling}. Then
	\[\E\Bigg(
	\E\bigg[ ( \dot{\eta}^\ell-\ddot{\eta}^\ell )^2
		\,\bigg|\,\dot{\tree}_{\vrt\crt} \bigg]^2
	\Bigg)
	\le \f{(\ddot{\alpha}-\dot{\alpha})^2}
		{ 2^{4k} / k^{O(1)} }\]
for all $\ell\ge0$. (On the left-hand side above, in the inner expectation, $\dot{\eta}^\ell$ is determined as a measurable function of $\dot{\tree}_{\vrt\crt}$. The outer expectation is with respect to the law of $\dot{\tree}_{\vrt\crt}\sim\PGW(\dot{\alpha})$.)

\begin{proof}
The bound certainly holds for $\ell=0$, where we have
$\dot{\eta}^\ell=\ddot{\eta}^\ell=1/2$ with probability one. Suppose inductively that the bound holds for $\ell-1\ge0$. Define the i.i.d.\ array
	\[\vec{t}\equiv\bigg(
		(t^\plus_{ij}, t^\minus_{ij})_{i,j\ge1}
		\bigg)
	\]
where each entry $t=(\dot{t},\ddot{t})$ is an independent sample from the monotone coupling $\PGW(\dot{\alpha},\ddot{\alpha})$. Let
	\beq\label{e:array.coupled.trees.a}
	\vec{a}
	\equiv \bigg( 
		(a^\plus_{ij},a^\minus_{ij})_{i,j\ge1} \bigg)
	\eeq
where each entry $a=(\dot{\eta},\ddot{\eta})$ is obtained by applying the map $\bm{F}_{\ell-1}$ to the corresponding entry in $\vec{t}$: that is,
	\[
	a^\plus_{ij}
	\equiv (\dot{\eta}^\plus_{ij},
	\ddot{\eta}^\plus_{ij})
	\equiv \Bigg(
		\Big[\bm{F}_{\ell-1}
		(\dot{t}^\plus_{ij})\Big](\minus),
		\Big[\bm{F}_{\ell-1}
		(\ddot{t}^\plus_{ij})\Big](\minus)
		\Bigg)
	\in[0,1)^2\,.
	\]
Next, let
$\dot{d}^\PM,\delta^\PM$ be independent random variables with
	\[
	\dot{d}^\PM
	\sim\Pois\bigg( \f{\dot{\alpha}k}{2}\bigg)\,,\quad
	\delta^\PM
	\sim\Pois\bigg( \f{
		(\ddot{\alpha}-\dot{\alpha})
		k}{2}\bigg)\,.
	\]
Let $\ddot{d}^\PM\equiv
\dot{d}^\PM+\delta^\PM$.
Let $\dot{T}\equiv\dot{\tree}_{\vrt\crt}$ be the variable-to-clause tree with $|\pd\vrt(\PM)\setminus\crt|=\dot{d}^\PM$, such that the $i$-th clause in $\pd\vrt(\PM)\setminus\crt$
has child subtrees $\dot{t}_{ij}$ for $1\le j\le k-1$.
Likewise, let $\ddot{T}\equiv\ddot{\tree}_{\vrt\crt}$ be the variable-to-clause tree with $|\pd\vrt(\PM)\setminus\crt|=\ddot{d}^\PM$, such that the $i$-th clause in $\pd\vrt(\PM)\setminus\crt$
has child subtrees $\ddot{t}_{ij}$ for $1\le j\le k-1$. Then 
	\[
	(\dot{T},\ddot{T})
	\equiv
	\Big(
	\dot{\tree}_{\vrt\crt},\ddot{\tree}_{\vrt\crt}
	\Big)
	\sim\PGW(\dot{\alpha},\ddot{\alpha})\,.\]
With $F(x,y)$ as in \eqref{e:def.F.x.y}, we can express
	\[
	\dot{\eta}^\ell
	= F(\dot{\Pi}^\minus,\dot{\Sigma})
	=\f{(1-\dot{\Pi}^\minus) \exp(\dot{\Sigma})}
	{1+(1-\dot{\Pi}^\minus) \exp(\dot{\Sigma})}
	\]
where we use similar notation as in Definition~\ref{d:distributional.var.recurs},
 \eqref{e:Pi.as.sum.of.X}, and
 \eqref{e:stability.notations}:
	\[
	\dot{\Pi}^\PM
	\equiv \prod_{i=1}^{\dot{d}^\PM}
	\dot{s}^\PM_i\,,\quad
	\dot{s}^\PM_i\equiv 
		1- \prod_{j=1}^{k-1} \dot{\eta}^\PM_{ij}\,,\quad
	\dot{\Sigma}^\PM_i
	\equiv -\log\dot{\Pi}^\PM\,,\quad
	\dot{\Sigma}\equiv
	\log \f{\dot{\Pi}^\plus}{\dot{\Pi}^\minus}\,.
	\]
We make all the analogous definitions to express $\ddot{\eta}^\ell
=F(\ddot{\Pi}^\minus,\ddot{\Sigma})$. As an intermediary between $\dot{T}$ and $\ddot{T}$, we define quantities that use the random messages $\ddot{\eta}^\PM_{ij}$
with the random degrees $\dot{d}^\PM$:
	\[
	\tilde{\Pi}^\PM
	\equiv \prod_{i=1}^{\dot{d}^\PM}
	\ddot{s}^\PM_i\,,\quad
	\ddot{s}^\PM_i\equiv 
		1- \prod_{j=1}^{k-1}
		\ddot{\eta}^\PM_{ij}\,,\quad
	\tilde{\Sigma}^\PM_i
	\equiv -\log\tilde{\Pi}^\PM
	\,,\quad
	\tilde{\Sigma}\equiv
	\log \f{\tilde{\Pi}^\plus}{\tilde{\Pi}^\minus}\,.
	\]
With this notation, we can decompose
	\beq\label{e:diff.eta.terms}
	\ddot{\eta}^\ell-\dot{\eta}^\ell
	=\overbrace{\bigg(
	F(\ddot{\Pi}^\minus,\ddot{\Sigma})
	-F(\tilde{\Pi}^\minus,\tilde{\Sigma})
	\bigg)}^{\textup{denote this }
		\mathbf{A}}
	+\overbrace{\bigg(
	F(\tilde{\Pi}^\minus,\tilde{\Sigma})
	-F(\dot{\Pi}^\minus,\dot{\Sigma})
	\bigg)}^{\textup{denote this }\mathbf{a}}
	\eeq
Let $\dot{X}^\PM_i\equiv-\log\dot{s}^\PM_i$
and $\ddot{X}^\PM_i\equiv-\log\ddot{s}^\PM_i$. Using the bound \eqref{e:F.bdd.derivs} from
Lemma~\ref{l:F.bounds.calculus}, we have
	\begin{align*}
	&\E\bigg[
	\mathbf{A}^2\,\bigg|\,\dot{T}\bigg]
	\le
	2\,\E\Bigg[
	\Big(\ddot{\Pi}^\minus-\tilde{\Pi}^\minus\Big)^2
	+\Big(\ddot{\Sigma}-\tilde{\Sigma}\Big)^2
	\,\Bigg|\,\dot{T}
	\Bigg]\\
	&\qquad=2\,\E\bigg[(\tilde{\Pi}^\minus)^2\,
			\bigg|\,\dot{T}
			\bigg]
		\E\Bigg[
		\bigg( 1 -
		\prod_{i=\dot{d}^\minus+1}
			^{\dot{d}^\minus+ \delta^\minus}
			 \ddot{s}^\minus_i
			 \bigg)^2 \,
			 \Bigg|\,\dot{T}
			 \Bigg]
	+ 2\,\E\Bigg[
		\bigg(
		\sum_{i=\dot{d}^\minus+1}
			^{\dot{d}^\minus+ \delta^\minus}
		\ddot{X}^\minus_i
		-
		\sum_{i=\dot{d}^\plus+1}
			^{\dot{d}^\plus+ \delta^\plus}
		\ddot{X}^\plus_i
		\bigg)^2
		\,\Bigg|\,
		\dot{T}\Bigg]\\
	&\qquad\le
	2\,\E\bigg[(\tilde{\Pi}^\minus)^2\,
			\bigg|\,\dot{T}
			\bigg]
		\P\Bigg( \Pois\bigg( 
		\f{k(\ddot{\alpha}-
		\dot{\alpha})}{2} \bigg)
		>0 \Bigg)
	+2\,\E\Big[ \Pois( k(\ddot{\alpha}
	-
		\dot{\alpha})) \Big]
		\E\Big[ (\ddot{X}^\minus_i)^2 \Big]\,.
	\end{align*}
Recall from Lemma~\ref{l:X.tail.bounds} that
$\ddot{X}^\minus_i$ is well concentrated around 
roughly $1/2^{k-1}$, from which it follows that
	\[\tilde{\Pi}^\minus
	=\f1{\exp(\Sigma^\minus)}
	=\exp\Bigg\{- \sum_{i=1}^{\dot{d}^\minus}
	 \ddot{X}^\minus_i
	\Bigg\}
	\]
is well concentrated around $1/2^k$. It follows that
	\[\E\Bigg(
	\E\Big[\mathbf{A}^2\,
		\Big|\,\dot{T}\Big]^2
		\Bigg)
	\le \f{(\ddot{\alpha}-\dot{\alpha})^2}
		{2^{4k}/k^{O(1)}}\,.\]
As for $\mathbf{a}$,
again using
the bound \eqref{e:F.bdd.derivs} from
Lemma~\ref{l:F.bounds.calculus}, we have
	\beq\label{eq-recursion-bound}
	\E\Big[ \mathbf{a}^2 \,\Big|\,
		\dot{T}\Big]
	\le 2\,\E\Bigg[(\dot{\Pi}^\minus-\tilde{\Pi}^\minus )^2
	+
	(\dot{\Sigma}-\tilde{\Sigma})^2
			\,\Bigg|\, \dot{T}
			\Bigg]\,.
	\eeq
For the first term on the right-hand side of \eqref{eq-recursion-bound} we can bound
	\[(\dot{\Pi}^\minus-\tilde{\Pi}^\minus)^2
	\le 
	\dot{d}^\minus
	\sum_{i=1}^{\dot{d}^\minus}
	(\Pi^\minus[i])^2
	(\dot{s}^\minus_i-\ddot{s}^\minus_i)^2\,,\quad
	\Pi^\minus[i] \equiv
		\prod_{i'=1}^{i-1}
		\dot{s}^\minus_{i'}
		\prod_{i'=i+1}^{\dot{d}^\minus}
		\ddot{s}^\minus_{i'}\,.\]
Then, using that $\dot{d}^\PM$ are measurable functions of $\dot{T}$, we have
	\[\E\Bigg(
		\E\bigg[
		(\dot{\Pi}^\minus-\tilde{\Pi}^\minus)^2
		\,\bigg|\,\dot{T}\bigg]^2
		\Bigg)
	\le \overbrace{\E\Bigg[
		(\dot{d}^\minus)^3
		\sum_{i=1}^{\dot{d}^\minus}
		(\Pi^\minus[i])^4
		\Bigg]}^{\le k^{O(1)}}
	\,\E\Bigg( 
	\E\bigg[ (\dot{s}^\minus_1-\ddot{s}^\minus_1)^2
			\,\bigg|\,
			\dot{T}\bigg]^2
			\Bigg)\,,\]
where the first factor is $\le k^{O(1)}$ because $\dot{d}^\minus$ is well concentrated around $k\dot{\alpha}/2$ while each $\Pi^\minus[i]$ is well concentrated around $1/2^k$
(using Lemma~\ref{l:X.tail.bounds} again). For the second factor, we can bound
	\[(\dot{s}^\minus_i-\ddot{s}^\minus_i)^2
	\le k\sum_{j=1}^{k-1}
		u^\minus_i[j]^2 
		(\dot{\eta}^\minus_{ij}
			-\ddot{\eta}^\minus_{ij} )^2\,,\quad
	u^\minus_i[j]
	\equiv \prod_{j'=1}^{j-1}
		\dot{\eta}^\minus_{ij'}
		\prod_{j'=j+1}^{k-1}
		\ddot{\eta}^\minus_{ij'}\,.\]
Combining with the inductive hypothesis gives
	\[\E\Bigg(\E\bigg[(\dot{s}^\minus_i-\ddot{s}^\minus_i)^2
		\,\bigg|\,\dot{T}
		\bigg]^2
		\Bigg)
	\le \f{(\ddot{\alpha}-\dot{\alpha})^2}
	{2^{4k} / k^{O(1)}}
		\Bigg\{\max_{1\le j\le k-1}
		\E[ u^\minus_i[j]^4 ]
		\Bigg\}
	\le \f{(\ddot{\alpha}-\dot{\alpha})^2}{
		2^{8k} / k^{O(1)}}\,.\]
Turning to the second term on the right-hand side of \eqref{eq-recursion-bound}, we have
	\[\E\Bigg(
	\E\bigg[ (\dot{\Sigma}^\minus
		-\tilde{\Sigma}^\minus)^2 
		\,\bigg|\, \dot{T} \bigg]^2
	\Bigg)
	=\E\Bigg(
	\E\Bigg[ 
	\Bigg(
	\sum_{i=1}^{\dot{d}^\minus} 
	(\dot{X}^\minus_i-
	\ddot{X}^\minus_i)
	\Bigg)^2 
		\,\Bigg|\, \dot{T} \Bigg]^2
	\Bigg)
	\le 
	\E\Big[ (\dot{d}^\minus)^2\Big]
	\E\Bigg(
	\E\bigg[ (\dot{X}^\minus_i-
	\ddot{X}^\minus_i)^2\,\bigg|\,\dot{T}\bigg]^2
	\Bigg)\,.\]
By a similar derivation as for \eqref{e:telescoping.sum.u.h}, we can bound
	 \[\E\Bigg(\E\bigg[
	(\dot{X}^\minus_1 -
		\ddot{X}^\minus_1)^2
	\,\bigg|\, \dot{T} \bigg]^2\Bigg)
	\le k^3
	\sum_{j=1}^{k-1}
	\E\Bigg(
	\bigg(\f{u[j]}{1-u[j]}\bigg)^4
	\Bigg)
	\,\E\Bigg(
	\E\bigg[(\eta_j-\tilde\eta_j)^2 \,\bigg|\, \dot{T}\bigg]^2
	\Bigg)
	\le \f{(\ddot{\alpha}-\dot{\alpha})^2}{
		2^{8k} / k^{O(1)}}\,,\]
where the last step again uses the inductive hypothesis together with Lemma~\ref{l:X.tail.bounds}.
Altogether we obtain
	\[
	\E\bigg(
		\E\Big[\mathbf{a}^2
		\,\Big|\,\dot{T}\Big]^2\bigg)
	\le 
	\f{(\ddot{\alpha}-\dot{\alpha})^2}{ 2^{4k}/k^{O(1)}}.\]
Combining with the bound on $\mathbf{A}$
verifies the inductive hypothesis, concluding the proof.
\end{proof}
\end{lem}

\begin{lem}\label{lem-coupling-Q-Q-prime}
Let $0\le \ddot{\alpha}-\dot{\alpha}\le\exp(-2^k)$, and $(\dot{\eta}^\ell,\ddot{\eta}^\ell)\sim \mu^\ell(\dot{\alpha},\ddot{\alpha})$. Then
	\[\E\bigg(
	\E\Big[ \ddot{\eta}^\ell-\dot{\eta}^\ell
		\,\Big|\, \dot{T} \Big]^2
	\Bigg)
	\le \f{(\ddot{\alpha}-\dot{\alpha})^2}
		{2^{3k}/k^{O(1)}}\]
for all $\ell\ge0$.

\begin{proof}
We will argue by induction, using the same notation as in the proof of Lemma~\ref{lem-Q-monotone-coupling-prelim}. The bound clearly holds for $\ell=0$, so suppose it holds for $\ell-1\ge0$. Decompose 
$\ddot{\eta}^\ell-\dot{\eta}^\ell =\mathbf{A}+\mathbf{a}$ as in \eqref{e:diff.eta.terms}. Applying
the bound \eqref{e:F.second.derivs} from Lemma~\ref{l:F.bounds.calculus} gives
	\[
	\mathbf{a} 
	=\overbrace{\bigg\{
	\Big( \tilde{\Pi}^\minus-\dot{\Pi}^\minus \Big)
	F_x(\dot{\Pi}^\minus,\dot{\Sigma})
	\bigg\}}^{\textup{denote this }\mathbf{a}_x}
	{}+{}
	\overbrace{\bigg\{
	(\tilde{\Sigma}-\dot{\Sigma})
	F_y(\dot{\Pi}^\minus,\dot{\Sigma})
	\bigg\}}^{\textup{denote this }\mathbf{a}_y}
	{}+{}
	\overbrace{O\bigg(
	(\tilde{\Pi}^\minus-\dot{\Pi}^\minus)^2
	\bigg)}^{\textup{denote this }\mathbf{a}_{xx}}
	{}+{}
	\overbrace{
	O\bigg(
	\Big(\tilde{\Sigma}-\dot{\Sigma}\Big)^2
	\bigg)
	}^{\textup{denote this }\mathbf{a}_{yy}}\,.
	\]
We deal with each term separately in what follows.\smallskip

\noindent\bemph{Bound on $\mathbf{a}_x$.} Let $\Pi^\minus[i]$ and $u^\minus_i[j]$ be as defined in the proof of Lemma~\ref{lem-Q-monotone-coupling-prelim}. Then we have
	\[\tilde{\Pi}^\minus-\dot{\Pi}^\minus
	=
	\sum_{i=1}^{\dot{d}^\minus}
	\Pi^\minus[i] (\ddot{s}^\minus_i
		-\dot{s}^\minus_i)
	=
	\sum_{i=1}^{\dot{d}^\minus}
	\Pi^\minus[i]
	\sum_{j=1}^{k-1}
	u^\minus_i[j]
	(\dot{\eta}^\minus_{ij}
	-\ddot{\eta}^\minus_{ij})\,.\]
Meanwhile, $F_x(\dot{\Pi}^\minus,\dot{\Sigma})$ is a measurable function of $\dot{T}$, 
and from Lemma~\ref{l:F.bounds.calculus} we have $|F_x(\dot{\Pi}^\minus,\dot{\Sigma})|\le1$. Combining with the inductive hypothesis, together with the independence of the random subtrees, we have
	\begin{align*}
	\E\bigg(\E\Big[\mathbf{a}_x\,\Big|\,\dot{T}
		\Big]^2\bigg)
	&\le\E\Bigg(
	\Bigg\{
	\sum_{i=1}^{\dot{d}^\minus}	\sum_{j=1}^{k-1}
	\E\bigg[
		\Pi^\minus[i]
		u^\minus_i[j]
		\,\bigg|\,\dot{T}\bigg]
		\E\bigg[
		(\dot{\eta}^\minus_{ij}
		-\ddot{\eta}^\minus_{ij})
		\,\bigg|\,\dot{T}\bigg]
	\Bigg\}^2
	\Bigg)
	\\
	&\le \f{(\ddot{\alpha}-\dot{\alpha})^2}
		{2^{3k}/k^{O(1)}}
	\E\Bigg(
	\dot{d}^\minus k
	\sum_{i=1}^{\dot{d}^\minus}	\sum_{j=1}^{k-1}
	\E\bigg[
		\Pi^\minus[i]
		u^\minus_i[j]
		\,\bigg|\,\dot{T}\bigg]^2
	\Bigg) \le \f{(\ddot{\alpha}-\dot{\alpha})^2}
		{2^{3k}/k^{O(1)}}\,,
	\end{align*}
where the last inequality uses the estimates from Lemma~\ref{l:X.tail.bounds}.\smallskip

\noindent\bemph{Bound on $\mathbf{a}_y$.} For this part of the proof we will generate the pair $(\dot{T},\ddot{T})$ in a slightly different way, as follows. We take the subtrees $\dot{t}_{ij}$ and $\ddot{t}_{ij}$ as before; the only difference is in the $\PM$-degrees of the root variable $\vrt$. Let $d\sim\Pois(k\dot{\alpha})$, and $\delta\sim\Pois(k(\ddot{\alpha}-\dot{\alpha}))$.
Let $S_i$ be i.i.d.\ symmetric random signs, and let
	\[
	\dot{d}^\PM
	\equiv
	\sum_{i=1}^d \Ind{S_i=\PM}\,,\quad
	\ddot{d}^\PM
	\equiv
	\sum_{i=1}^{d+\delta} \Ind{S_i=\PM}\,.
	\]
We then let $\vrt$ have $\PM$-degrees $\dot{d}^\PM$ in $\dot{T}$, and $\ddot{d}^\PM$ in $\ddot{T}$. Then, similarly to \eqref{e:telescoping.sum.u.h} we can decompose
	\[
	\tilde{\Sigma}
	-\dot{\Sigma}
	=\sum_{i=1}^d S_i
		(\ddot{X}_i - \dot{X}_i)
	= -\sum_{i=1}^d S_i
		\sum_{j=1}^{k-1}
		\log\f{1-u_i[j]\ddot{\eta}_{ij}}
		{1-u_i[j]\dot{\eta}_{ij}}\,.
	\]
Note in the above that $\dot{X}_i$ stands for $\dot{X}^\plus_i$ or $\dot{X}^\minus_i$ depending on the sign $S_i$, but we have dropped the superscripts for notational convenience. Applying the approximation
	\[
	f(\eta)
	=\log\f1{1-u_i[j]\eta}
	=f(\eta_0)
	+(\eta-\eta_0)f'(\eta_0)
	+O\bigg( \|f''\|_\infty
	(\eta-\eta_0)^2 \bigg)
	\]
gives the following expansion for $\mathbf{a}_y$:
	\[\mathbf{a}_y
	=
	\overbrace{\Bigg(
	-F_y(\dot{\Pi}^\minus,\dot{\Sigma})
		\sum_{i=1}^d S_i
		\sum_{j=1}^{k-1}
		\f{u_i[j] (\ddot{\eta}_{ij}-\dot{\eta}_{ij})}
		{1-u_i[j]\dot{\eta}_{ij}}
		\Bigg)}^{\textup{denote
			 this }
			\mathbf{a}_{y,1}}
	{} + {} O\overbrace{\Bigg(\sum_{i=1}^d
		\sum_{j=1}^{k-1}
		\Big(\f{u_i[j]
		(\ddot{\eta}_{ij}-\dot{\eta}_{ij})
		}{1-u_i[j]}\Big)^2
		\Bigg)}^{\textup{denote
			 this }
			\mathbf{a}_{y,2}}\,.
	\]
We bound these terms separately, beginning with the quadratic term $\mathbf{a}_{y,2}$:
	\[
	\E[\mathbf{a}_{y,2}\,|\,\dot{T}]^2
	\le dk
	\sum_{i=1}^d
	\sum_{j=1}^{k-1}
	\E\Bigg[
		\bigg(\f{u_i[j]}{1-u_i[j]}\bigg)^2
		\,\Bigg|\, \dot{T}\Bigg]^2
	\,
	\E\Big[(\ddot{\eta}_{ij}-\dot{\eta}_{ij})^2
			\,\Big|
			\,\dot{T}_{ij}\Big]^2\,.
	\]
Combining with the inductive hypothesis,
Lemma~\ref{l:X.tail.bounds}, and Lemma~\ref{lem-Q-monotone-coupling-prelim} gives
	\[\E\bigg(\E[\mathbf{a}_{y,2}\,|\,\dot{T}]^2
	\bigg)
	\le\f{(\ddot{\alpha}-\dot{\alpha})^2}{2^{4k} / k^{O(1)}}\,.\]
We next bound the contribution from the linear term $\mathbf{a}_{y,1}$, which requires some more care. Expanding $\E[\E[\mathbf{a}_{y,1}\,|\,\dot{T}]^2]$ as a double sum over indices $1\le i_1,i_2\le d$, we first bound the contribution from the cross terms $i_1\ne i_2$. Here the issue is that although $S_{i_1}S_{i_2}$ certainly has mean zero, 
 the cross term may not have mean zero because of the factor $F_y(\dot{\Pi}^\minus,\dot{\Sigma})^2$ which is correlated with $S_{i_1},S_{i_2}$. However we shall argue that this effect is negligible. To this end, we shall approximate $\dot{\Pi}^\minus$ and $\dot{\Sigma}$ by
	\[\dot{\Pi}^\minus[i_1i_2]
	\equiv \prod_{\substack{i \in
		[\dot{d}]
		\setminus \set{i_1,i_2},\\
		S_i=\minus }}
		\dot{s}_i, \quad
	\dot{\Sigma}[i_1i_2]
	\equiv
	\sum_{\substack{i \in
		[\dot{d}]
		\setminus \set{i_1,i_2} }}
		S_i X_i\,.\]
Let $H(x,y) \equiv F_y(x,y)^2$, and note that
for $0\le x\le1$ and $xe^y\le1$ we have
	\begin{align*}
	|H_x(x,y)| 
	&= 2 \cdot \Bigg| 
	\f{(1-x)e^y}{(1+(1-x)e^y)^3}
	\f{1-(1-x)e^y}{1+(1-x)e^y}
	\f{e^y}{(1-xe^y)+ e^y}
	\Bigg|
	\le2\,,\\
	|H_y(x,y)|
	&= 2 \cdot\Bigg|
	\bigg(\f{(1-x)e^y}{(1+(1-x)e^y)^2}\bigg)^2
	\f{1-(1-x)e^y}{1+(1-x)e^y}
	\Bigg|
	\le2\,.
	\end{align*}
It follows that, abbreviating 
$H[i_1i_2]
	\equiv
	H(\dot{\Pi}^\minus[i_1i_2],\dot{\Sigma}[i_1i_2])$, we have
	\[
	\Big|H(\dot{\Pi}^\minus,\dot{\Sigma})
	-H[i_1i_2]\Big|
	\le
	2\Bigg\{
	\Big|\dot{\Pi}^\minus-\dot{\Pi}^\minus[i_1i_2]\Big|
	+ 
	\Big|\dot{\Sigma}-\dot{\Sigma}[i_1i_2]\Big|
	\Bigg\}\,.
	\]
Returning to the expression for $\mathbf{a}_{y,1}$, 
let us denote
	\[\bm{T}(i)
	\equiv\sum_{j=1}^{k-1}
	\E\bigg[
	\f{u_i[j]}{1-u_i[j]
		\dot{\eta}_{ij} }
		\,\bigg|\,\dot{T} \bigg]\,
	\E\Big[
	(\ddot{\eta}_{ij}-\dot{\eta}_{ij})
	\,\Big|\, \dot{T}
	\Big]\,.\]
We then use the Cauchy--Schwarz inequality to bound
	\[
	\bm{T}(i)^2
	\le
	k \sum_{j=1}^{k-1} 
	\overbrace{\Bigg(
	\E\bigg[
	\f{u_i[j]}{1-u_i[j]}
		\,\bigg|\,\dot{T} \bigg]^2
	\,
	\E\Big[
	(\ddot{\eta}_{ij}-\dot{\eta}_{ij})
	\,\Big|\, \dot{T}
	\Big]^2
	\Bigg)}^{\textup{denote this } \bm{T}(ij)^2}\,.
	\]
Using the inductive hypothesis and Lemma~\ref{l:X.tail.bounds}, the $i_1,i_2$ cross term in $\E[\E[\mathbf{a}_{y,1}\,|\,\dot{T}]^2]$ can be bounded by
	\begin{align*}
	&\overbrace{\E\Bigg(
		H[i_1i_2]
	\prod_{l=1}^2 \Big\{ S_{i_l} \bm{T}(i_l)\Big\}
		\Bigg)}^\text{zero}
	+
	\E\Bigg[
	\Big( H(\dot{\Pi}^\minus,
	\dot{\Sigma}) - H[i_1i_2]\Big)
	\prod_{l=1}^2 \Big\{ S_{i_l} \bm{T}(i_l)\Big\}
		\Bigg] \\
	&=O(1)
	\E\Bigg[
	\bigg\{
	\Big|\dot{\Pi}^\minus-\dot{\Pi}^\minus[i_1i_2]\Big|
	+\Big|\dot{\Sigma}-\dot{\Sigma}[i_1i_2]\Big|
	\bigg\}
	\sum_{l=1}^2
	\bm{T}(i_l)^2
	\Big]\Bigg]\\
	&\le
	\f{ (\ddot{\alpha}-\dot{\alpha})^2 }
		{2^{3k} / k^{O(1)}}
	\E\Bigg[
	\bigg(
	\Big|\dot{\Pi}^\minus-\dot{\Pi}^\minus[i_1i_2]\Big|
	+\Big|\dot{\Sigma}-\dot{\Sigma}[i_1i_2]\Big|
	\bigg)
	\sum_{l=1}^2
	\sum_{j=1}^{k-1}
	\E\bigg[
	\f{u_{i_l}[j]}{1-u_{i_l}[j]}
		\,\bigg|\,\dot{T}\setminus\dot{T}_{ij}
		\bigg]^2
	\Bigg]
	\le\f{ (\ddot{\alpha}-\dot{\alpha})^2 }
		{2^{6k} / k^{O(1)}}\,.
	\end{align*}
The total contribution to
$\E[\E[\mathbf{a}_{y,1}\,|\,\dot{T}]^2]$
from cross terms $i_1\ne i_2$ is thus
$ \le k^{O(1)}(\ddot{\alpha}-\dot{\alpha})^2 / 2^{4k}$, so
	\begin{align*}
	&\E\bigg(\E[\mathbf{a}_{y,1}\,|\,\dot{T}]^2\bigg)
	\le\f{ (\ddot{\alpha}-\dot{\alpha})^2 }
		{2^{4k} / k^{O(1)}}
		+ \E\Bigg(
		F_y(\dot{\Pi}^\minus,\dot{\Sigma})^2
		\sum_{i=1}^d 
		\E\bigg[
		\sum_{j=1}^{k-1}
		\f{u_i[j] (\ddot{\eta}_{ij}-\dot{\eta}_{ij}) }
			{1-u_i[j]\dot{\eta}_{ij}} 
		\,\bigg|\,\dot{T}
		\bigg]^2
	\Bigg) \\
	&\le
	\f{(\ddot{\alpha}-\dot{\alpha})^2}
		{2^{4k}/k^{O(1)}}
	+ k\, \E\Bigg(
		\sum_{i=1}^d 
		\sum_{j=1}^{k-1}
		\E\bigg[
		\f{u_i[j] }
			{1-u_i[j]} 
		\,\bigg|\,\dot{T}
		\bigg]^2
		\E\Big[(\ddot{\eta}_{ij}-\dot{\eta}_{ij})
		\,\Big|\,\dot{T}\Big]^2
	\Bigg)
	\le \f{(\ddot{\alpha}-\dot{\alpha})^2}
		{2^{4k}/k^{O(1)}}\,,
	\end{align*}
having again used the inductive hypothesis
together with Lemma~\ref{l:X.tail.bounds}.

\smallskip\noindent\bemph{Bounds
	on $\mathbf{a}_{xx}$ and $\mathbf{a}_{yy}$.}
The quadratic terms are more straightforward to bound:
	\begin{align*}
	\E\Big[
	\E[\mathbf{a}_{xx}\,|\,\dot{T}]^2\Big]
	&\le k^{O(1)}
	\E\Bigg(
	(\dot{d}^\minus)^3
	\sum_{i=1}^{\dot{d}^\minus}
	\sum_{j=1}^{k-1}
	\E\Big[
		\Pi^\minus[i]^2
		u_i[j]^2
		\,\Big|\,
		\dot{T}\Big]^2
	\,
	\E\Big[(\ddot{\eta}^\minus_{ij}
		-\dot{\eta}^\minus_{ij})^2\,
	\Big|\,\dot{T}\Big]^2
	\bigg)\,,\\
	\E\Big[
	\E[\mathbf{a}_{yy}\,|\,\dot{T}]^2\Big]
	&\le k^{O(1)}
	\E\Bigg(
	d^3 \sum_{k=1}^d \sum_{j=1}^{k-1}
	\E\bigg[
		\Big(
		\f{u_i[j]}{1-u_i[j]} \Big)^2
		\,\bigg|\, \dot{T}\bigg]^2
	\,\E\Big[(\ddot{\eta}^\minus_{ij}
		-\dot{\eta}^\minus_{ij})^2\,
	\Big|\,\dot{T}\Big]^2
	\Bigg)\,.
	\end{align*}
Combining with
Lemma~\ref{l:X.tail.bounds} and 
 Lemma~\ref{lem-Q-monotone-coupling-prelim} gives
	\[
	\E\Big[
	\E[\mathbf{a}_{xx}\,|\,\dot{T}]^2\Big]
	+\E\Big[
	\E[\mathbf{a}_{yy}\,|\,\dot{T}]^2\Big]
	\le \f{(\ddot{\alpha}-\dot{\alpha})^2}
		{ 2^{4k} / k^{O(1)} }\,.
	\]
This concludes our bounds for $\mathbf{a}$.\smallskip

\noindent\bemph{Bound on $\mathbf{A}$.}
Recall that the tree $\dot{T}$
has root degree $\dot{d}$
while $\ddot{T}$
has root degree $\ddot{d}\equiv \dot{d} + \delta$.
Let $\tilde{T}$ be the subtree of $\ddot{T}$
formed by deleting all subtrees descended
from root neighbors
with indices larger than $\dot{d}$,
so that $\dot{T}\subseteq\tilde{T}\subseteq\ddot{T}$.
By Jensen's inequality,
	\[
	\E\bigg(
	\E[\mathbf{A}\,|\,\dot{T}]^2\bigg)
	\le \E\bigg(
	\E[\mathbf{A}\,|\,\tilde{T}]^2\bigg)\,.
	\]
If $\ddot{d}=\dot{d}$ then clearly $\mathbf{A}=0$. Since $|\mathbf{A}|\le1$ almost surely, we can bound
	\[\E\bigg(
	\E\Big[\mathbf{A}\Ind{\delta\ge2}
		\,\Big|\,\tilde{T}\Big]^2\bigg)
	\le \P(\delta\ge2)^2
	\le k^{O(1)}(\ddot{\alpha}-\dot{\alpha})^4\,.\]
On the event $\delta=1$, applying the bound
\eqref{e:F.second.derivs} from Lemma~\ref{l:F.bounds.calculus} gives
	\begin{align*}
	\E\Big[\mathbf{A}\Ind{\delta=1}
		\,\Big|\,\tilde{T}\Big]
	&=\E\Bigg[
	\overbrace{
	\bigg\{
	F_x(\tilde{\Pi}^\minus,\tilde{\Sigma})
	\Big(\ddot{\Pi}^\minus -\tilde{\Pi}^\minus\Big)
	\bigg\}
	}^{\textup{denote this }\bar{\mathbf{A}}_x}
	+\overbrace{
	\bigg\{
	F_y(\tilde{\Pi}^\minus,\tilde{\Sigma})
	\Big(\ddot{\Sigma}-
		\tilde{\Sigma}\Big)
	\bigg\}}^{\textup{denote this }
		\bar{\mathbf{A}}_y}\\
	&\qquad\qquad\qquad+
	\underbrace{
	O\bigg(
	(\ddot{\Pi}^\minus -
		\tilde{\Pi}^\minus)^2
	\bigg)}_{\textup{denote this }\bar{\mathbf{A}}_{xx}}
	+\underbrace{
	O\bigg(
	(\ddot{\Sigma}^\minus -
		\tilde{\Sigma}^\minus)^2
	\bigg)}_{\textup{denote this }\bar{\mathbf{A}}_{yy}}
	\Bigg]\,.
	\end{align*}
Since $\delta=1$
with probability $ O(k) (\ddot{\alpha}-\dot{\alpha})$,
we find
	\[\E\bigg(\E\Big[\bar{\mathbf{A}}_x
		\Ind{\delta=1}\,\Big|\,\tilde{T}\Big]^2
		\bigg)
	\le \f{(\ddot{\alpha}-\dot{\alpha})^2
		\E[ (\tilde\Pi^\minus)^2 ]}
		{ 2^{2k} / k^{O(1)} }
	\le \f{(\ddot{\alpha}-\dot{\alpha})^2}
		{ 2^{4k} / k^{O(1)} }\,,\]
and by similar considerations
	\[
	\E\bigg(
		\E\Big[\bar{\mathbf{A}}_{xx}
		\Ind{\delta=1}\,\Big|\,\tilde{T}\Big]^2
		\bigg)
	+\E\bigg(\E\Big[\bar{\mathbf{A}}_{yy}
		\Ind{\delta=1}\,\Big|\,\tilde{T}\Big]^2
		\bigg)
	\le \f{(\ddot{\alpha}-\dot{\alpha})^2}
		{2^{4k} / k^{O(1)}}\,.
	\]
Finally, by symmetry,
$\E[\bar{\mathbf{A}}_y\,|\,\tilde{T}]=0$.
Altogether this yields
	\[
	\E\Big[
	\E[\mathbf{A}\,|\,\dot{T}]^2
	\Big]
	\le 
	k^{O(1)}(\ddot{\alpha}-\dot{\alpha})^4
	+\f{(\ddot{\alpha}-\dot{\alpha})^2}
		{ 2^{4k} / k^{O(1)} }
	\le\f{(\ddot{\alpha}-\dot{\alpha})^2}
		{ 2^{4k} / k^{O(1)} }\,,
	\]
where the last step uses the assumed bound on $\ddot{\alpha}-\dot{\alpha}$.

\smallskip\noindent\textbf{Conclusion.}
Substituting the above estimates into \eqref{e:diff.eta.terms} gives
the claimed bound.
\end{proof}
\end{lem}

Recall from Proposition~\ref{p:fp} that $\mu^\ell(\alpha)$ converges weakly to $\mu(\alpha)$ as $\ell\to\infty$. Under the monotone coupling, it is clear that we also have $\mu^\ell(\dot{\alpha},\ddot{\alpha})$ converging to a well-defined limit $\mu(\dot{\alpha},\ddot{\alpha})$. If $(\dot{\eta},\ddot{\eta})$ is sampled from $\mu(\dot{\alpha},\ddot{\alpha})$, then $\dot{\eta}$ has marginal law $\mu(\dot{\alpha})$ while $\ddot{\eta}$ has marginal law $\mu(\ddot{\alpha})$. 

\begin{cor}\label{cor-X-hat-X}
Let $0\le \ddot{\alpha}-\dot{\alpha} \le \exp(-2^k)$.
Let $\vec{a}$ be as in \eqref{e:array.coupled.trees.a}, except that the entries are independent samples from $\mu(\dot{\alpha},\ddot{\alpha})$ rather than from
$\mu^{\ell-1}(\dot{\alpha},\ddot{\alpha})$. Then, with the same notation as before, we have
	\begin{align*}
	\E\bigg(
	\E[\ddot{X}_i-\dot{X}_i \,|\,\dot{T} ]^2\bigg)
	&\le \f{(\ddot{\alpha}-\dot{\alpha})^2}
		{2^{5k} / k^{O(1)}}\,,\\
	\E\bigg(
		\E\Big[ 
		(\ddot{X}_i-\dot{X}_i)^2
		\,\Big|\,\dot{T} \Big]^2
		\bigg)
	&\le \f{(\ddot{\alpha}-\dot{\alpha})^2}
		{2^{8k} / k^{O(1)}}\end{align*}
for all $1\le i\le \dot{d}=\dot{d}^\plus+\dot{d}^\minus$.

\begin{proof}
With the same notation as before, we can expand
	\[
	\ddot{X}_i-\dot{X}_i
	=\overbrace{\Bigg\{
		\sum_{j=1}^{k-1}
		\f{u_i[j](\ddot{\eta}_{ij}-\dot{\eta}_{ij})}
			{1-u_i[j]\dot{\eta}_{ij}}
			\Bigg\}}^{\textup{denote this }\mathbf{x}_1}
	{}+{}\overbrace{
	O\Bigg( \sum_{j=1}^{k-1}
		\Big(\f{u_i[j](\ddot{\eta}_{ij}
			-\dot{\eta}_{ij})}
			{1-u_i[j]}\Big)^2
		\Bigg)
	}^{\textup{denote this }\mathbf{x}_2}\,.
	\]
For the quadratic term,
applying Lemma~\ref{lem-Q-monotone-coupling-prelim} 
(in the limit $\ell\to\infty$) gives
	\[\E\Big(
		\E[\mathbf{x}_2\,|\,\dot{T}]^2\Big)
	\le \f{(\ddot{\alpha}-\dot{\alpha})^2}{
		 2^{8k}/k^{O(1)}}\,.
	\]
For the linear term, applying
 Lemma~\ref{lem-coupling-Q-Q-prime} 
 gives
	\[
	\E\Big(\E[\mathbf{x}_1\,|\,\dot{T}]^2\Big)
	\le k\,
	\E \Bigg( \sum_{j=1}^{k-1}
		\E\bigg[\f{u[j]}{1-u[j]}
			\,\bigg|\, \dot{T}\bigg]^2
		\,\E\Big[ \ddot{\eta}_j-\dot{\eta}_j
			\,\Big|\, \dot{T}\Big]^2
			\Bigg)
	\le \f{ (\ddot{\alpha}-\dot{\alpha})^2}
		{2^{5k} /k^{O(1)}}\,.
	\]
Combining these gives the first assertion.
Similar calculations give
	\[
	\E\Big(\E[ (\ddot{X}_i-\dot{X}_i)^2
		\,|\,\dot{T} ]^2\Big)
	\le
	k^{O(1)}
	\E \Bigg( \sum_{j=1}^{k-1}
		\E\bigg[ \Big(\f{u[j]}{1-u[j]}\Big)^2
			\,\bigg|\, \dot{T}\bigg]^2
		\,\E\Big[ (\ddot{\eta}_j-\dot{\eta}_j)^2
			\,\Big|\, \dot{T}\Big]^2\Bigg)
	\le\f{ (\ddot{\alpha}-\dot{\alpha})^2}
		{2^{8k} /k^{O(1)}}\,.
	\]
This gives the second assertion.
\end{proof}
\end{cor}

\begin{proof}[Proof of Proposition~\ref{p:phi}]
Let $\albd\le \dot{\alpha}\le\ddot{\alpha}\le\aubd$ with $\ddot{\alpha}-\dot{\alpha} \le \exp(-2^k)$.
Our goal is to upper bound the difference
$\Phi(\ddot{\alpha})-\Phi(\dot{\alpha})$ where $\Phi$ is the $\onersb$ free energy defined by 
\eqref{e:phi.alpha}. To this end, define the function
	\[
	G(x,y)
	\equiv \log\bigg(
	\f1{e^x}+\f1{e^y}-\f1{e^{x+y}}
	\bigg)\,,
	\]
and note that $\Phi(\alpha)=\Phi_1(\alpha)+(k-1)\Phi_2(\alpha)$ where
	\begin{align}\nonumber
	\Phi_1(\alpha)
	&\equiv
	\E \log(
	\f1{\exp(\Sigma^\plus)}
	+\f1{\exp(\Sigma^\minus)}
	-\f1{\exp(\Sigma^\plus+\Sigma^\minus)}
		)
	= \E G(\Sigma^\plus,\Sigma^\minus)\,,\\
	\label{e:def.X.prime}
	\Phi_2(\alpha)
	&\equiv -\alpha\,\E\log 
		\bigg( 1- \prod_{j=1}^k \eta_j\bigg)
	\equiv \alpha \, \E X'\,.
	\end{align}
We define $\tilde{\Sigma}^\PM$ as in the proof of Lemma~\ref{lem-Q-monotone-coupling-prelim}, and decompose (similarly to \eqref{e:diff.eta.terms})
	\[
	\Phi_1(\ddot{\alpha})
	-\Phi_1(\dot{\alpha})
	=\overbrace{\E \bigg(
		G(\ddot{\Sigma}^\plus,\ddot{\Sigma}^\minus)
		-G(\tilde{\Sigma}^\plus,\tilde{\Sigma}^\minus)
		\bigg)}^{\textup{denote this }\mathbf{G}}
	{}+{}\overbrace{\E\bigg(
	G(\tilde{\Sigma}^\plus,\tilde{\Sigma}^\minus)
	-G(\dot{\Sigma}^\plus,\dot{\Sigma}^\minus)
	\bigg)}^{\textup{denote this }\mathbf{g}}\,.
	\]
We further decompose
	\[
	\mathbf{g}
	=
	\E \Bigg( \overbrace{\bigg\{
	G(\tilde{\Sigma}^\plus,\tilde{\Sigma}^\minus)
	-G(\tilde{\Sigma}^\plus,\dot{\Sigma}^\minus)
	\bigg\}}^{\textup{denote this }\mathbf{g}^\minus}
	+\overbrace{\bigg\{
	G(\tilde{\Sigma}^\plus,\dot{\Sigma}^\minus)
	-G(\dot{\Sigma}^\plus,\dot{\Sigma}^\minus)
	\bigg\}}^{\textup{denote this }\mathbf{g}^\plus}
	\Bigg)\,.
	\]
It is easily checked that
$|G_x|\le1$ and $|G_{xx}|\le1$, so
	\[
	\mathbf{g}^\plus
	=\overbrace{\Bigg\{
	\sum_{i=1}^{\dot{d}^\plus}
		G_x\Big(\Sigma^\plus[i] + \dot{X}^\plus_i,
			\dot{\Sigma}^\minus\Big)
		\cdot
		\Big( \ddot{X}^\plus_i - \dot{X}^\plus_i \Big)
		\Bigg\} }^{\textup{denote this }\mathbf{g}_x}
	{}+{}\overbrace{
	O\Bigg( \sum_{i=1}^{\dot{d}^\plus}
		\Big(
		\ddot{X}^\plus_i - \dot{X}^\plus_i
		\Big)^2\Bigg)}
		^{\textup{denote this }\mathbf{g}_{xx}}.
	\]
Applying Corollary~\ref{cor-X-hat-X}
gives
	\[
	(\E\mathbf{g}^\plus)^2
	\le 2\,
	\E\bigg(
	\E[ \mathbf{g}_x\,|\,\dot{T}]^2
	+\E[ \mathbf{g}_{xx}\,|\,\dot{T}]^2
	\bigg)
	\le 
	 \f{(\ddot{\alpha}-\dot{\alpha})^2}
		{2^{3k}/k^{O(1)}}
	+ \f{(\ddot{\alpha}-\dot{\alpha})^2}
		{2^{6k}/k^{O(1)}}
	\le \f{(\ddot{\alpha}-\dot{\alpha})^2}
		{2^{3k}/k^{O(1)}}\,.
	\]
Similar bounds hold for $\mathbf{g}^\minus$, so we conclude
	\[|\mathbf{g}|
	=\bigg|
	\E\Big(\mathbf{g}^\minus+\mathbf{g}^\minus\Big)
	\bigg|
	\le
	\f{\ddot{\alpha}-\dot{\alpha}}
		{2^{3k/2}/k^{O(1)}}\,.\]
It remains to estimate $\mathbf{G}$. This is similar to the analysis of $\mathbf{A}$
in the proof of Lemma~\ref{lem-coupling-Q-Q-prime}: write $\delta\equiv\ddot{d}-\dot{d}=\delta^\plus+\delta^\minus$ for the difference in root degrees between $\dot{T}$ and $\ddot{T}$. When $\delta=0$ we have simply $\mathbf{G}=0$. The expected contribution from the event $\delta\ge2$ is negligible. It remains to consider the case $\delta=1$. Without loss of generality, take the case that $\delta^\plus=1$ and $\delta^\minus=0$: 
	\begin{align*}
	\mathbf{G}_{1,0}
	&\equiv\E\Bigg[
	G(\ddot{\Sigma}^\plus,\ddot{\Sigma}^\minus)
	-G(\tilde{\Sigma}^\plus,\tilde{\Sigma}^\minus)
	; (\delta^\plus,\delta^\minus)=(1,0)\Bigg]\\
	&=\E\Bigg[
	G(\tilde{\Sigma}^\plus
		+ \ddot{X}_{\ddot{d}^\plus},
		\tilde{\Sigma}^\minus)
	-G(\tilde{\Sigma}^\plus,\tilde{\Sigma}^\minus)
	; (\delta^\plus,\delta^\minus)=(1,0)\Bigg]\\
	&= -\f12 \cdot
	\f{k (\ddot{\alpha}-\dot{\alpha})}{2} \cdot
	\f{1+O(2^{-k/20})}
			{2^{k-1}}\,,
	\end{align*}
where the last estimate uses Lemma~\ref{l:X.tail.bounds}
and the fact that the partial derivative $G_x(\tilde{\Sigma}^\plus,\tilde{\Sigma}^\minus)$
 is well concentrated around $-1/2$. Altogether we conclude
 	\[\Phi_1(\ddot{\alpha})
	-\Phi_1(\dot{\alpha})
	= - \f{k(\ddot{\alpha}-\dot{\alpha}) }{2^k
		} 
		\Bigg\{ 1 + \f{O(1)}{2^{k/20}}
	\Bigg\}\,.
	\]
Next, let $\dot{X}'$ and $\ddot{X}'$ be defined analogously to $X'$ in 
\eqref{e:def.X.prime}, but using the messages $\dot{\eta}_j$ and $\ddot{\eta}_j$. Then
	\[
	\Phi_2(\ddot{\alpha})-\Phi_2(\dot{\alpha})
	=
	\overbrace{\bigg\{
	(\ddot{\alpha}-\dot{\alpha})
	\cdot \E\dot{X}'
	\bigg\}}^{\textup{denote this }\mathbf{h}_{21}}
	+
	\overbrace{\bigg\{
	\ddot{\alpha}\,\E(\ddot{X}'-\dot{X}')
	\bigg\}}^{\textup{denote this }\mathbf{h}_{22}}
	\,.
	\]
A trivial modification of Corollary~\ref{cor-X-hat-X}
gives 
	\[
	(\E\mathbf{h}_{22})^2
	\le
	\f{\ddot{\alpha}^2
	(\ddot{\alpha}-\dot{\alpha})
	}{2^{5k}/k^{O(1)}}
	\le\f{(\ddot{\alpha}-\dot{\alpha})
	}{2^{3k}/k^{O(1)}}\,.
	\]
On the other hand, a trivial modification of bounds from Lemma~\ref{l:X.tail.bounds} (with $k$ in place of $k-1$) gives
	\[\E\mathbf{h}_{21}
	= \f{\ddot{\alpha}-\dot{\alpha}}{2^k}
	\Bigg\{ 1 + \f{O(1)}{2^{k/20}}
	\Bigg\}
	\,.\]
Combining the above estimates gives finally
	\[
	\Phi(\ddot{\alpha})
		-\Phi(\dot{\alpha})
	=\Bigg\{
	 - \f{k(\ddot{\alpha}-\dot{\alpha}) }{2^k
		} 
	+\f{(k-1)(\ddot{\alpha}-\dot{\alpha})}{2^k}
	\Bigg\}
	\Bigg\{ 1 + \f{O(1)}{2^{k/20}}
	\Bigg\}
	= - \f{\ddot{\alpha}-\dot{\alpha} }{2^k
		} 
		\Bigg\{ 1 + \f{O(k)}{2^{k/20}}
	\Bigg\}\,.
	\]
This proves that $\Phi$ is strictly decreasing 
on the interval $\albd\le\alpha\le\aubd$, as desired.
\end{proof}

\pagebreak

\bibliographystyle{alphaabbr}

{
\raggedright
\bibliography{refs,refs-conf}
}

\end{document}